\documentclass[a4paper,twoside]{article}
\usepackage[latin1]{inputenc} 

\usepackage[T1]{fontenc}
\usepackage[french]{babel}
\usepackage{fullpage}
\usepackage{amssymb}
\usepackage{amsmath}
\usepackage{stmaryrd} 
\usepackage{mathrsfs} 
\usepackage{graphicx}
\usepackage[all]{xy}
\usepackage{tikz}
\usetikzlibrary{calc} 
\usepackage{xcolor}
\usepackage{enumerate}
\usepackage{dsfont} 

\usepackage{manfnt}

\usepackage{fancyhdr} 
\usepackage{datetime} 
\usepackage[rightlabels]{titletoc}

\titlecontents{section}[0pt]{\addvspace{10pt\filright}}
              {\bf \contentspush{\thecontentslabel\ }}
              {}{\titlerule*[0em]{.} \bf \thecontentspage}
              
\titlecontents{subsection}[20pt]{\addvspace{0.2pt\filright}}
              {\contentspush{\thecontentslabel\ }}
              {}{\titlerule*[0.8em]{.}\thecontentspage}

\usepackage{amsthm}
\newtheoremstyle{CTstyleI}{\topsep}{\topsep}{\itshape}{}{}{\bf.}{.5em}{%
   \thmnumber{{\bfseries #2}\bf .}\thmname{ \bfseries #1}\normalfont\thmnote{ \bf (#3)}} 
\theoremstyle{CTstyleI}
\newtheorem{thmCQCT}{}[subsection]
\newtheorem{prop}[thmCQCT]{Proposition}
\newtheorem{fact}[thmCQCT]{Fait}
\newtheorem{theo}[thmCQCT]{Théorème}
\newtheorem{coro}[thmCQCT]{Corollaire}
\newtheorem{defn}[thmCQCT]{Définition}
\newtheorem{defns}[thmCQCT]{Définitions}
\newtheorem{lem}[thmCQCT]{Lemme}
\newtheorem{conj}[thmCQCT]{\'Enoncé conjectural}
\newtheorem {scholie}[thmCQCT]{Scholie} 

\newtheoremstyle{CTstyleII}{\topsep}{\topsep}{}{}{}{\bf.}{.5em}{%
                     \thmnumber{{\bfseries #2}\bf .}\thmname{ \bfseries #1}\normalfont\thmnote{ \bf (#3)}} 
\theoremstyle{CTstyleII}
\newtheorem{rmq}[thmCQCT]{Remarque}
\newtheorem{rmqs}[thmCQCT]{Remarques}  
\newtheorem{exemple}[thmCQCT]{Exemple}

\input{0-CodeBordermatrix.tex}
\renewcommand \preceq {\preccurlyeq} 
\renewcommand \succeq {\succcurlyeq}
\renewcommand \ge {\geqslant} 
\renewcommand \le {\leqslant}

\newcommand \mouton {{\includegraphics[height=.6em, trim=0 6 0 6]{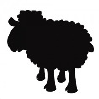}}}
\newcommand \pointu{\begin{tikzpicture}\draw [->] (0,0) -- (0.2,0.2) -- (0.4,0) ;\end{tikzpicture}}
\newcommand \cercle[1]{\tikz[baseline=(char.base)]{
            \node[shape=circle,draw,inner sep = 0.4pt] (char){#1};}}

\newcommand \Heti[1] {\omit \quad \mbox{\footnotesize$#1$} \hfil}
\newcommand \Veti[1] {{\rotatebox{90}{\mbox{\footnotesize$#1$}}}}
\newcommand \VR {\kern-\arraycolsep \strut\vrule &\kern-\arraycolsep}
\newcommand \HR[1] {\noalign{\vskip-8pt}\hrulefill\span\multispan{#1}\cr \noalign{\vskip-4pt}}

\newcommand \VRepaisse {\kern-\arraycolsep \strut\vrule\vrule &\kern-\arraycolsep}
\def\hligne#1{\leaders\hrule height#1\hfill}
\newcommand \HRepaisse[1] {\noalign{\vskip-8pt} \hligne{1pt}\span\multispan{#1}\cr \noalign{\vskip-4pt}}

\newcommand \bfA {\mathbf A}
\newcommand \bfB {\mathbf B}

\newcommand \bfc {\mathbf c}
\newcommand \bfe {\mathbf e}

\newcommand \bff {\mathbf f}

\newcommand \bfh {\mathbf h}
\newcommand \bfk {\mathbf k}
\newcommand \bfK {\mathbf K}
\newcommand \bfl {\mathbf l}
\newcommand \bfm {\mathbf m}
\newcommand \bfn {\mathbf n}

\newcommand \bfR {\mathbf R}
\newcommand \bfS {\mathbf S}
\newcommand \bfu {\mathbf u}
\newcommand \bfv {\mathbf v}
\newcommand \bfw {\mathbf w}

\newcommand \bfx {\mathbf x}
\newcommand \bfy {\mathbf y}
\newcommand \bfz {\mathbf z}

\newcommand \bsnu {\boldsymbol\nu}
\newcommand \bsmu {\boldsymbol\mu}
\newcommand \bsp{\boldsymbol{p}}
\newcommand \bsq{\boldsymbol{q}}
\newcommand \bsr{\boldsymbol{r}}
\newcommand \bss{\boldsymbol{s}}
\newcommand \pXD {\bsp\uX^D}

\newcommand \bsnabla{\boldsymbol\nabla}
\newcommand \bsOmega{\boldsymbol\Omega}
\newcommand \bsomega{\boldsymbol\omega}

\newcommand \bbA {\mathbb A}
\newcommand \bbB {\mathbb B}
\newcommand \bbC {\mathbb C} 
\newcommand \bbM {\mathbb M}
\newcommand \bbN {\mathbb N}
\newcommand \bbP {\mathbb P}
\newcommand \bbQ {\mathbb Q}
\newcommand \bbR {\mathbb R}
\newcommand \bbS {\mathbb S} 
\newcommand \bbZ {\mathbb Z}

\newcommand \Jex {\mathfrak J}
\newcommand \Jid {\Jex_{1\setminus2,d}^{(i)}}
\newcommand \fa {\mathfrak a}
\newcommand \fb {\mathfrak b}
\newcommand \fB {\mathfrak B}
\newcommand \fC {\mathfrak C}
\newcommand \fD {\mathfrak D}
\newcommand \fm {\mathfrak m}
\newcommand \fp {\mathfrak p}
\newcommand \fS {\mathfrak S}
\newcommand \brick {\mathfrak B} 

\newcommand \ua {{\underline a}}
\newcommand \ub {{\underline b}}
\newcommand \uc {{\underline c}}
\newcommand \uF {{\underline F}}
\newcommand \uG {{\underline G}}
\newcommand \uL {{\underline L}}
\newcommand \um {{\underline m}}
\newcommand \uP {{\underline P}} 
\newcommand \uQ {{\underline Q}}
\newcommand \uq {{\underline q}} 
\newcommand \uR {{\underline R}}
\newcommand \us {{\underline s}}
\newcommand \uT {{\underline T}}
\newcommand \uu {{\underline u}}
\newcommand \uU {{\underline U}}
\newcommand \uv {{\underline v}}
\newcommand \uw {{\underline w}}
\newcommand \ux {{\underline x}}
\newcommand \uy {{\underline y}}
\newcommand \uX {{\underline X}}
\newcommand \uxi {{\underline\xi}}
\newcommand \uY {{\underline Y}}
\newcommand \uz {{\underline z}}
\newcommand \uzero {{\underline 0}}

\newcommand \calB {\mathcal B}
\newcommand \calC {\mathcal C}
\newcommand \calD {\mathcal D}
\newcommand \calE {\mathcal E}
\newcommand \calF {\mathcal F}
\newcommand \calG {\mathcal G}
\newcommand \calH {\mathcal H}

\newcommand \calL {\mathcal L}
\newcommand \calM {\mathcal M}
\newcommand \calN {\mathcal N}
\newcommand \calO {\mathcal O}
\newcommand \calP {\mathcal P}
\newcommand \calQ {\mathcal Q}
\newcommand \calR {\mathcal R}
\newcommand \calS {\mathcal S}
\newcommand \calT {\mathcal T}

\newcommand \rmc {\mathrm c} 
\newcommand \rmC {\mathrm C} 
\newcommand \rmH {\mathrm H}
\newcommand \rmK {\mathrm K}
\newcommand \Mmac {\mathrm M}
\newcommand \Smac {\mathrm S} 

\newcommand \Cech {\v Cech}
\newcommand \vH {\check{\rmH}}
\newcommand \vC {\check{\rmC}}

\newcommand \longleftmapsto{\longleftarrow\!\shortmid}

\newcommand \BW {\bigwedge\nolimits}
\newcommand \Hom {\mathop{\mathrm{Hom}}\nolimits}
\newcommand \idest {\textit{i.e.}}
\newcommand \scp[2] {\langle#1\mid#2\rangle}

\DeclareRobustCommand \Im {\mathop{\mathrm{Im}}}
\newcommand \Ker {\mathop{\mathrm{Ker}}}
\newcommand \Coker {\mathop{\mathrm{Coker}}}
\newcommand \transpose[1]{{\,^{\rm t}#1}}

\newcommand \Ann {{\mathrm{Ann}}}
\newcommand \diag {\mathrm{diag}}
\newcommand \eval {\text{\rm éval}}
\newcommand \Gr {\mathop{\mathrm{Gr}}}
\newcommand \Tr {\mathop{\mathrm{Tr}}}

\newcommand \Id {\mathrm{Id}}
\newcommand \id {\mathrm{id}}

\newcommand \pgcd {\mathop{\mathrm{pgcd}}}
\newcommand \poids {\mathrm{poids}}

\newcommand \Discr {\mathop{\mathrm{Disc}}\nolimits_{\rm r}} 
\newcommand \Disc {\mathop{\mathrm{Disc}}} 
\newcommand \Res {\mathrm{Res}}
\newcommand \res {\mathrm{res}}
\newcommand \omegares{\omega_\res}
\newcommand \omegaRes[1]{\omega_{\res,#1}}
\newcommand \omegaresP {\omega_{\res,\uP}}
\newcommand \omegarespXD {\omega_{\res,\pXD}}
\newcommand \omegaresPred {\omega_{\res,\uPred}}
\newcommand \uomegares{\overline\omega_\res}

\newcommand \sat {\mathrm{sat}}
\newcommand \ElimIdealL {\langle\uL\rangle^\sat\cap\bfA}
\newcommand \evalxi {\eval_{\uxi}}
\newcommand \coeff {\mathrm{coeff}}

\newcommand \DivSeq {\mathrm{Div}}
\newcommand \minDiv {\mathrm{minDiv}}
\newcommand \smin {\sigma\rule[0.7mm]{0.1cm}{0.3pt}\mathrm{min}{}}
\newcommand \rhomin {\rho\rule[0.7mm]{0.1cm}{0.3pt}\mathrm{min}{}} 
\newcommand \sminDiv {\sigma\rule[0.7mm]{0.1cm}{0.3pt}\minDiv{}}

\newcommand \uPdelta {\langle\uP\rangle_\delta}
\newcommand \uPred {\uP^{\rm red}}
\newcommand \uPsat {\langle\uP\rangle^\sat}
\newcommand \ElimIdeal {\uPsat\cap\bfA}
\newcommand \indetsPi {\text{indets pour les $P_i$}}
\newcommand \emouton {{D-\mathds 1}}
\newcommand \MoutonNoir {\textit{mouton-noir}}

\newcommand \vertex[1]{{\substack{#1\\\bullet}}} 
\newcommand \exposant {\mathrm{exposant}} 

\newcommand \Un {{\mathds 1}} 
\newcommand \Jac {\mathrm{Jac}}

\newcommand \Syl {\mathscr{S}} 

\newcommand \dsV {\mathds{V}}
\newcommand \dsL {\mathds{L}}
\newcommand \dsU {\mathds{U}} 

\newcommand \Cramer {\mathrm{Cramer}}
\newcommand \CayleyVect {\fC^\star}
\newcommand \DVect {\boldsymbol{\mathcal D}^\star}
\newcommand \FittVect {\boldsymbol{\mathcal F}^\star}
\newcommand \MacRae {\mathfrak R}
\newcommand \MacRaeVect {\boldsymbol{\MacRae^\star}}

\newcommand \Det {\mathop{\mathrm {Det}}\nolimits}
\newcommand \oriented[1] {\boldsymbol{[} #1 \boldsymbol{]}} 
\newcommand \kappastar {\kappa^\star}

\newcommand \gen {\mathrm{gen}}

\newcommand \Last {\mathrm{Last}}
\newcommand \crochet[1] {^{(#1)}}
\newcommand \Ph {\uP^{(h)}}
\newcommand \Xh {\uX^{(h)}}
\newcommand \calME {{\calM\double\calE}} 

\newcommand \double[1] {\llbracket #1 \rrbracket}
\newcommand \maxDiv {\mathrm{maxDiv}}
\newcommand \maxDivSh {\mathrm{maxDivSh}}
\newcommand \End {\mathrm{End}}
\newcommand \EndSh {\mathrm{EndSh}}
\newcommand \Shift {\mathrm{Shift}}
\newcommand \sh {\mathrm{sh}}
\newcommand \ssup {\triangleright}
\newcommand \eq {\vDash}
\newcommand \scrE {\mathscr E} 

\newcommand \scrV {\mathscr V} 
\newcommand \scrW {\mathscr W} 
\newcommand \scrN {\mathscr N} 

\newcommand \fSwap[2]{f^{#1\to#2}}
\newcommand \Serie {\mathrm{S\acute{e}rie}}

\newcommand \Niun {\mathscr N_{i,1}}
\newcommand \Nideux {\mathscr N_{i,2}}
\newcommand \Nndeux {\mathscr N_{n,2}} 
\newcommand \Aik {A_{i,k}}

\newcommand\sw {{\rm sw}}
\newcommand\AXdSod{\bfA[\uX]_d^{\calS_{0,d}}}
\newcommand\AXdSodprime{\bfA[\uX]_d^{\calS_{0,d'}}}
\newcommand\BdSod{\bfB_d^{\calS_{0,d}}}
\newcommand\BdSodprime{\bfB_d^{\calS_{0,d'}}}
\newcommand\swDet{\Det^\sw}
\newcommand\swbsOmega{\bsOmega^\sw}

\newcommand \Ind {\mathrm{Ind}}
\newcommand \Elt {\mathrm{Elt}}
\newcommand \DivG {\overleftarrow{\mathrm{Div}}}
\newcommand \XbJ {X^\beta X^{D(J)}}
\newcommand \JexRJ {\Jex^R[J]}
\newcommand \Rev[1] {\overleftarrow{#1}}
\newcommand \RevE {\overleftarrow{E}}
\newcommand \Jdecomposition {$J$-décomposition}
\newcommand \Rdecomposition {$R$-décomposition}

\newcommand \moins {\kern-1.5pt - \kern-1.5pt} 
\newcommand \hderniers {\mathbf{h}} 
\newcommand \kderniers {\mathbf{k}}

\newcommand \uPX {\uP(\uX)}
\newcommand \uPY {\uP(\uY)}
\newcommand \PXminusY {\uP(\uX) - \uP(\uY)}
\newcommand \Bez {\mathscr{B}} 
\newcommand \varpidres {\varpi^\res_d}
\newcommand \varpidRes[1] {\varpi^\res_{d,#1}}
\newcommand \bez {\mathscr{A}} 
\newcommand \bezendo {\mathrm{Bez}} 
\newcommand \matrice {\mathrm{Mat}} 

\newcommand \intd {\ \llcorner\ } 
\newcommand \Hd {\mathbf H} 

\newcommand \sbullet {{\scriptscriptstyle\bullet}}

\newcommand \idealx{\langle\ux\rangle}
\newcommand \ideala{\langle\ua\rangle}
\newcommand \idealap{\langle\ua'\rangle}
\newcommand \gotimes{\sideset{^g}{}{\mathop\otimes}} 

\usepackage{scrtime}
\def\includefilename{\jobname}

\usepackage{makeidx}
\makeindex

\index{completement@complètement sécante|see{suite complètement sécante}}
\index{reguliere@régulière|see{suite régulière}}
\index{saturé|see{idéal saturé}}
\index{selection@sélection|see{mécanisme de sélection}}
\index{Sylvester|see{application de Sylvester}}

\usepackage[colorlinks, draft=false, linkcolor=blue]{hyperref}

\begin{document}
\renewcommand{\proofname}{Preuve.}


\pagenumbering{Roman}   


\title{\bf Une approche combinatoire pour le fondement \\ du résultant multivarié \\
      {\normalsize Le résultant multivarié pour les enfants motivés}}

\author{Claude Quitté, Claire Tête\\[2mm]
claude.quitte@orange.fr,  
claire.tete@bginette.fr}

\begin{titlepage}

\maketitle

\begin{center}
\vspace{3cm}
\includegraphics[height=6cm]{0-MoutonNoir}
\vspace{1cm}
{\Large
$$
\bfA[\uX]_\delta \ = \  \langle X_1^{d_1}, \dots, X_n^{d_n} \rangle_\delta   \ \oplus \, 
                 \bfA {\includegraphics[height=.7em, trim=0 4 0 5]{0-MoutonNoir}}
$$
}
\end{center}

\thispagestyle{empty}
\end{titlepage}


\setcounter{tocdepth}{2} 

L'illustration en page de couverture est due à Claire. 
Essayons d'expliquer pourquoi, dans cette page, c'est \og la fête à la
\fbox {dimension 1}\fg.

\smallskip
Il intervient des $d_i$ qui sont des entiers $\ge 1$. On dit que
$D := (d_1, \ldots, d_n)$ est un \emph {format de degrés}. Il y a
également un indice $\delta$ nommé \emph{degré critique de $D$}:
$$
\delta := \sum_i (d_i-1)
$$
Pour ce degré critique, il y a \fbox {un} et \fbox {un seul} monôme de $\bfA[\uX] = \bfA[X_1, \dots, X_n]$,
de degré $\delta$, qui n'est divisible par aucun des~$X_i^{d_i}$. Etant le seul,
c'est un \MoutonNoir. Le voici
$$
X^\emouton = X_1^{d_1-1} \cdots X_n^{d_n-1} = \mouton
$$
D'où évidemment un quotient de dimension \fbox{1}
$$
\bfA[\uX]_\delta\ /\ \langle X_1^{d_1}, \ldots, X_n^{d_n}\rangle_\delta = \bfA X^\emouton
\qquad\qquad
\dim \bfA[\uX]_\delta\ /\ \langle X_1^{d_1}, \ldots, X_n^{d_n}\rangle_\delta =
\fbox {1}
$$
Ce qui précède relève d'une combinatoire \emph {monomiale} facile.

\bigskip

Introduisons maintenant un système $\uP = (P_1, \cdots, P_n)$
de $\bfA[\uX]$ de format $D$, ce qui signifie que $P_i$ est un
polynôme homogène de degré $d_i$. Et posons
$$
\bfB_\delta = 
\bfB_\delta(\uP) = \bfA[\uX]_\delta\ /\ \langle P_1, \ldots, P_n\rangle_\delta 
$$
Alors, modulo une hypothèse\footnote{$\uP$ est une suite régulière.}
sur $\uP$, ce quotient $\bfB_\delta$ est un $\bfA$-module de rang
\fbox {1}.  Ceci ne signifie pas qu'il est libre de rang 1, mais que
l'on dispose\footnote{Que l'on peut expliciter si on en a la volonté.}
d'une famille finie \emph {fidèle} $\us = (s_\ell)$ de scalaires de~$\bfA$
telle qu'en localisant en chaque $s_\ell$, le $\bfA[1/{s_\ell}]$-module
$\bfB_\delta[1/s_\ell]$ est libre de rang \fbox {1}.
Fidèle signifie que l'annulateur de $\us$ est réduit à 0: $\Ann(\us) = 0$.

\bigskip

Mais $\bfB_\delta$ est bien plus qu'un $\bfA$-module de rang 1: c'est un
$\bfA$-module \fbox{de MacRae} de rang \fbox{1}. Ce n'est pas le lieu
d'expliquer ici ce que cela signifie, c'est l'objet de certains des
chapitres qui suivent.  On peut juste dire que $\bfB_\delta$
porte une forme linéaire (la dimension 1 est présente car il faut penser à une
forme $c$-linéaire alternée avec $\boxed{c=1}$):
$$
\uomegares = {\uomegares}_\uP : \bfB_\delta \to \bfA
$$
que l'on peut, si l'on veut, remonter  au niveau $\bfA[\uX]_\delta$
$$
\omegares = \omegaRes{\uP} : \bfA[\uX]_\delta \to \bfA
$$
Cette forme linéaire bien précise (et pas définie à un inversible
près!), polynomiale en les coefficients des~$P_i$, est calculable. Et
elle \emph {contient en son sein} le résultant $\Res(\uP)$ au sens où
$\Res(\uP)$ est une évaluation précise de $\omegares$: on verra que
$\Res(\uP) = \omegares(\nabla)$ où $\nabla \in \bfA[\uX]_\delta$ est
un déterminant bezoutien de $\uP$.  Précisons qu'un déterminant
bezoutien de $\uP$ est un déterminant $\det(\dsV)$ où
$\dsV \in \bbM_n(\bfA[\uX])$ est une matrice homogène telle que $\uP
= \uX.\dsV$.  En fait, le ``on verra'' n'est pas tout-à-fait approprié
puisque l'on prendra $\omegares(\nabla)$ comme définition du résultant
de~$\uP$.

\bigskip

Est ce que l'on aurait pu réaliser quelque chose d'analogue en degré
$d \ge \delta+1$? Oui et non. Oui au sens où $\bfB_d = \bfB_d(\uP)
= \bfA[\uX]_d\ /\ \langle P_1, \ldots, P_n\rangle_d$ est de MacRae,
mais de MacRae de rang 0. En tant que tel, il porte une forme
$0$-linéaire alternée, en clair un scalaire. Impossible, à partir d'un
tel $\bfB_d$, de retrouver l'information contenue dans $\omegares$,
objet bien plus riche qu'un scalaire. Donc non. Quid du
scalaire porté par~$\bfB_d$?  C'est le résultant de~$\uP$. Et en ce
qui concerne $\bfB_d$ pour $d < \delta$?  Cela sera raconté par
la suite.

\newpage
\tableofcontents

\newcommand \cdpinclude[1]{\cleardoublepage\include{#1}}



\setlength{\footnotesep}{11pt}

\section*{Introduction et guide de lecture}
\addcontentsline{toc}{section}{\bf Introduction et guide de lecture}
\pagenumbering{roman} 

L'histoire de cette étude a commencé en 2012 pour Claude Quitté, 2014
pour Claire Tête.  Après une dizaine d'années, le moment tant
attendu est venu d'écrire une introduction.

Inutile de faire croire qu'au départ tout était prévu comme il en est
d'usage dans les projets planifiés.  C'est tout le contraire: notre
parcours a été parfois, et même souvent, chaotique.

\medskip

Le titre mentionne l'adjectif ``combinatoire''.
Il s'agit là d'un aspect partiel\footnote{%
Aspect certes partiel mais si récréatif pour les auteurs.

\begin{minipage}{0.76\linewidth}
Comment ne pas évoquer les moments ludiques comme la chasse aux entiers
(strictement) positifs où l'on peut faire joujou avec les proches?
Une première solution (partielle) avait été trouvée par
Mireille et Patrice le 14 Janvier 2015. Il s'agit du problème suivant:
sur un cercle orienté, on dispose de $n$ entiers $x_0, \dots, x_{n-1}$
ce qui permet pour tout $i$, de considérer le successeur $x_{i+1}$ de $x_i$.
On suppose $\sum_i x_i = 0$.
S'il y a un $x_i > 0$, on le diminue d'une unité en augmentant d'une
unité son successeur, ce qui ne change pas la nullité de la somme.
De manière formelle: $x_i \leftarrow x_i- 1$, $x_{i+1} \leftarrow x_{i+1} + 1$. Il s'agit
de démontrer que ce processus se termine i.e. à un moment donné, qu'il n'y a plus
de $x_i > 0$; ce qui signifie que tous les $x_i$ sont~$\le 0$ et, puisque
leur somme est nulle, qu'ils sont tous nuls.

Ce résultat de terminaison s'intègre dans une famille de résultats
analogues développés dans le chapitre~\ref{ChapJeuCirculaire}.
 On déduit de ces résultats
un certain nombre de propriétés importantes de régularité de $\bfB := \bfA[\uX]/\langle\uP\rangle$,
qui font l'objet du chapitre~\ref{ChapBgenerique}.

\end{minipage}
\begin{minipage}{0.2\linewidth}
\begin {tikzpicture}
\def \a {20}
\def \r {0.70}
\draw[thick, >=latex, ->, blue] (\r*cos{(\a)}, \r*sin{(\a)}) arc (\a : 360+\a : \r) ;
\draw (0:\r) node[right] {$x_0$} ;
\draw (60:\r) node[above right] {$x_1$} ;
\draw (120:\r) node[above left] {$x_2$} ;
\draw (180:\r) node[left] {$x_3$} ;
\draw (240:\r) node[below left] {$x_4$} ;
\draw (300:\r) node[below right] {$x_5$} ;
\end {tikzpicture}
\end{minipage}
}
de notre travail. Comme ce qualificatif est inhabituel dans la
thématique du résultant, nous fournirons dans cette introduction
plusieurs exemples relevant de la combinatoire, exemples qui nous
paraissent significatifs. C'est le cas par exemple de la
chasse aux entiers positifs figurant en note de bas de page.

\medskip

Le sous-titre, ``Le résultant multivarié pour les enfants motivés'',
témoigne de la volonté de présenter des approches relativement
élémentaires, bien qu'il ne soit pas toujours facile pour les auteurs d'en
juger.\footnote{Le statut de l'adjectif ``élémentaire'' n'est
absolument pas clair: élémentaire pour qui?  En tout cas, ce dont nous
sommes sûrs, c'est qu'il n'y a point de suites spectrales dans notre
affaire.}

\medskip

Mentionnons d'emblée, quitte à décevoir un certain nombre de
collègues, que toute considération géométrique est absente de notre
traitement. Notre propos est purement algébrique et nous
parlerons beaucoup de l'idéal d'élimination sans parler vraiment
d'élimination.

\medskip

D'autre part, point important que nous tenons à
souligner: nous avons voulu que notre étude soit effective, \emph
{réellement} effective. Nous ne cachons pas qu'un certain nombre
de points ont été implémentés. Pour cette raison, dans notre étude, il
n'y a pas de raisonnements par l'absurde, non pas par principe,
mais parce qu'ils sont totalement incompatibles avec la
programmation. En fait, il n'y a pratiquement pas de négations,
d'énoncés négatifs (pour la même raison).
Dans le même ordre d'idées, nous apprécions tout particulièrement les propos de Demazure à la fin
de l'introduction de son article~\cite{Demazure2} \og rédaction~538
pour Bourbaki sortie du frigidaire\fg:

{\sl
Le développement de l'Algèbre moderne, commencé il y a près d'un
siècle, a renvoyé les anneaux de polynômes au statut de cas
particulier et les méthodes propres aux polynômes, comme la
{\it Théorie de l'élimination}, au conservatoire.
Mais \og les objets sont têtus\fg {} et les méthodes explicites ne
cessent de ressurgir. Un calcul est toujours plus général que le
cadre théorique dans lequel on l'enferme à une période donnée.
}


\medskip

La théorie du résultant que nous exposons ici est sans doute
échafaudée de manière non classique, en tout cas \emph {non} basée sur
la monogénéité de l'idéal d'élimination en terrain générique; certains
jugeront peut-être non orthodoxes les procédés utilisés par endroits.
Nous verrons, pourquoi et comment, nous en sommes venus là.

\medskip

Malgré le flou des souvenirs, on peut considérer que
le volet ci-dessous a constitué notre première
``réussite'' et nous a encouragés à poursuivre.

\subsection*{\og Relations binomiales déterminantales\fg{} (premier succès)}
\label {IntroBinomialBW}
\addcontentsline{toc}{subsection}{Relations binomiales déterminantales (premier succès)}

D'une manière radicalement différente des approches habituelles, nous
avons mis en place la formule de Macaulay, celle qui donne le
résultant d'un $n$-système $\uP$ en degré $d \ge \delta+1$ où $\delta
:= \sum_i (d_i-1)$ est le degré critique de~$\uP$ avec
$d_i = \deg(P_i)$. Nous notons $D := (d_1, \dots, d_n)$ le format
de degrés de $\uP = (P_1, \dots, P_n)$.

\smallskip

Mais plus que cela! En réalité, nous avons obtenu d'une part des égalités
analogues \emph {en tout degré $d$} dans lesquelles interviennent
certains \og dénominateurs\fg{} explicités comme déterminants.
D'autre part, dans notre étude, un certain nombre d'objets font
intervenir la relation d'ordre habituelle sur $\{1..n\}$.
Nous \og tordons\fg{} cette relation d'ordre à l'aide de
n'importe quelle permutation $\sigma \in \fS_n$,
en définissant une nouvelle relation d'ordre $<_\sigma$:
$$
\sigma(1) <_\sigma \sigma(2) <_\sigma \cdots <_\sigma \sigma(n)
$$

Si $O$ est un objet élaboré à partir de la relation d'ordre standard sur $\{1..n\}$,
on peut le déformer en un objet $O^\sigma$ défini à l'aide de~$<_\sigma$.
Ce procédé conduit à une famille de formules indexée par le groupe
symétrique $\fS_n$. Ces formules englobent celles pour le résultant
que nous écrivons avec nos notations\footnote{
$W_{1,d}(\uP)$ n'est pas une matrice mais un \emph{endomorphisme} d'un
$\bfA$-module libre de rang fini. Il n'y a donc pas de base à ordonner
convenablement ni de souci de signe $\pm$. En conséquence, les égalités
que nous mettons au point sont de vraies égalités $a=b$ et non des
pseudo-égalités $a=\pm b$. Le $\bfA$-module libre
en question est la composante homogène~$\Jex_{1,d}$ de degré~$d$
de l'idéal monomial $\Jex_1 := \langle X_1^{d_1}, \dots, X_n^{d_n}\rangle$.
Idem pour $W_{2,d}(\uP)$ relativement à un idéal monomial $\Jex_2$.
Les endomorphismes~$W^\sigma_{1,d}(\uP)$, $W^\sigma_{2,d}(\uP)$ sont les versions
tordues par~$\sigma\in\fS_n$. Tous ces endomorphismes s'obtiennent à partir
de l'application de Sylvester $\Syl_d(\uP)$:
$$
\Syl_d(\uP) : \bfA[\uX]_{d-d_1} \times \cdots \times \bfA[\uX]_{d-d_n} \to \bfA[\uX]_d,
\qquad
(U_1,\dots,U_n) \mapsto \sum_i U_iP_i
$$
}
$$
\Res(\uP) = \dfrac {\det W_{1,d}(\uP)} {\det W_{2,d}(\uP)} =
\dfrac {\det W^\sigma_{1,d}(\uP)} {\det W^\sigma_{2,d}(\uP)}
\qquad\qquad
d \ge \delta+1,
\qquad
\sigma \in \fS_n
\leqno (\deg\ge\delta+1)
$$
et fournissent également une expression de la
forme linéaire fondamentale $\omegaRes{\uP}
: \bfA[\uX]_\delta \to \bfA$:
$$
\omegaRes{\uP}(\sbullet) = \dfrac {\det\Omega_\uP(\sbullet)} {\det W_{2,\delta}(\uP)} =
\dfrac {\det\Omega^\sigma_\uP(\sbullet)} {\det W^\sigma_{2,\delta}(\uP)}
\qquad\qquad  \sigma \in \fS_n
\leqno (\deg=\delta)
$$
Cette forme linéaire \og porte\fg{} le résultant au sens où $\Res(\uP)$
est une évaluation explicite de $\omegaRes{\uP}$, cf.~page~\pageref{IntroDefPrivilegieeResultant}.

Cerise sur le gâteau: pour $d$ fixé, lorsque $\sigma$ varie dans
$\fS_n$ (ou seulement dans une partie bien ciblée de cardinal $n$ de
$\fS_n$), la famille des dénominateurs est de profondeur $\ge 2$
pourvu que $\uP$ soit en adéquation avec le jeu étalon $\uX^D =
(X_1^{d_1}, \dots, X_n^{d_n})$. 

\medskip

Pour établir ces égalités, nous avons fait intervenir le complexe
composante homogène de degré $d$ du complexe de Koszul de $\uX^D$, que
nous notons $\rmK_{\sbullet,d}(\uX^D)$, et mis au point une
combinatoire basique, du type positionnement dans un ensemble
totalement ordonné, qui s'est révélée fructueuse.

Voici quelques aspects de cette combinatoire qui prend sa
source dans le format de degrés $D = (d_1,\dots,d_n)$.

\begin {enumerate}[a)]
\item
A la base, il y a la notion \emph{d'indice de divisibilité} d'un monôme $X^\alpha$ i.e.
tout indice $i\in \llbracket 1,n\rrbracket$ tel que $X_i^{d_i} \mid X^\alpha$.
Et nous définissons son ensemble de divisibilité
$\DivSeq(X^\alpha) := \{i\in \llbracket 1,n\rrbracket \mid \alpha_i \ge d_i\}$.
Cette notion
fait naître une famille d'idéaux monomiaux $(\Jex_h)_{1\le h\le n}$,
dits excédentaires (notion introduite par nous-mêmes). Par exemple:
$$
\Jex_1 = \langle X_1^{d_1}, \dots, X_n^{d_n}\rangle, \qquad\qquad
\Jex_2 = \langle X_i^{d_i}X_j^{d_j} \mid 1 \le i < j\le n \rangle
$$
La plupart du temps, nous utilisons un
\emph {mécanisme de sélection}, procédé qui consiste à
choisir dans l'ensemble de divisibilité de $X^\alpha$, lorsqu'il est
non vide, un indice de divisibilité privilégié, par exemple le plus
petit, que nous notons $\minDiv(X^\alpha)$.

\item
Dans notre approche, le caractère monomial des modules, provenant de
la définition même des sous-modules de Macaulay $\Mmac_{k,d}$, joue un
rôle important. Au départ, pas de système $\uP$, uniquement le jeu
étalon $\uX^D = (X_1^{d_1}, \dots, X_n^{d_n})$, son complexe de Koszul
$\rmK_\sbullet(\uX^D)$ et ses composantes homogènes~$\rmK_{\sbullet,d}(\uX^D)$.
Les sous-modules de Macaulay
$\Mmac_{k,d}$ sont des sous-modules monomiaux (monôme ici a le sens
de monôme extérieur $X^\alpha e_{i_1} \wedge \cdots \wedge e_{i_k}$) de
$\rmK_{k,d}(\uX^D)$ \emph {associés au mécanisme de sélection $\minDiv$}
(contrairement aux idéaux $\Jex_h$ indépendants de tout mécanisme de
sélection).

Un sous-module monomial possède un supplémentaire canonique, à savoir
son supplémentaire monomial.  Cette propriété en apparence banale est
fondamentale car elle permet de mettre en place l'induit-projeté d'un
endomorphisme, notion qui interviendra à de nombreuses reprises.

\item
Le supplémentaire monomial de $\Mmac_k$ est noté $\Smac_k$; il est
canoniquement monomialement isomorphe à~$\Mmac_{k-1}$ (lorsque $k \ge 1$)
via l'induit-projeté $\beta_k(\uX^D)$ de $\partial_k(\uX^D)$ défini ci-dessous:
$$
\partial_k(\uX^D) \ =\ 
\NorthEastBordermatrix{ 
\scriptstyle\Mmac_k  &\scriptstyle\Smac_k& \\
\star & \cercle{$\beta_k$} &\Heti{\scriptstyle\Mmac_{k-1}} \\
\noalign{\vskip4pt}
\star & \star           &\Heti{\scriptstyle\Smac_{k-1}} \\
}
\qquad\qquad
\beta_k(\uX^D) : \Smac_k \overset{\simeq}{\longrightarrow} \Mmac_{k-1}
$$
En conséquence, nous disposons d'une décomposition du complexe $\rmK_\sbullet(\uX^D)$
dite de Macaulay:
$$
\def \egal{\rotatebox{90}{=}}
\xymatrix @H=0pt @R=2pt @C=4pc {
\rmK_n & \rmK_{n-1} &  \cdots  & \rmK_2 & \rmK_1  & \rmK_0  \\
\egal & \egal & & \egal & \egal & \egal \\
\Mmac_n=0 & \Mmac_{n-1} &  \cdots  & \Mmac_2 & \Mmac_1  & \Mmac_0  \\
\oplus & \oplus & & \oplus & \oplus & \oplus \\
\Smac_n \ar[uur]^{\beta_n(\uX^D)}_\simeq & \Smac_{n-1} & \cdots & \Smac_2 \ar[uur]^{\beta_2(\uX^D)}_\simeq
& \Smac_1 \ar[uur]^{\beta_1(\uX^D)}_\simeq & \Smac_0 \\
}
$$

\item
Une étude poussée de la décomposition de Macaulay de
$\rmK_{\sbullet,d}(\uX^D)$ nous a permis de décomposer chaque
sous-module de Macaulay $\Mmac_{k-1,d}$, pour $k \ge 1$,  à l'aide
des idéaux monomiaux excédentaires $(\Jex_{h,d})_{h\ge k}$,
cf. la proposition \ref{DecompositionMkminus1ParJh}.
\end {enumerate}

\bigskip

Lorsqu'on considère ensuite un système $\uP$ de format $D$, un certain
miracle se produit concernant les différentielles
$\partial_{k,d}(\uP)$ du complexe~$\rmK_{\sbullet,d}(\uP)$, complexe
dont les termes sont ceux de $\rmK_{\sbullet,d}(\uX^D)$.  Sans entrer
dans les détails, la \emph {première} différentielle
$\partial_{1,d}(\uP)$ donne naissance, via un procédé d'induit-projeté
couplé avec le mécanisme de sélection $\minDiv$, à des
endomorphismes~$W_{h,d}(\uP)$ de~$\Jex_{h,d}$, avec~$h \ge 1$.
Quant aux différentielles $\partial_{k,d}(\uP)$, chacune 
donne naissance à un induit-projeté $\beta_{k,d}(\uP) : \Smac_{k,d} \to \Mmac_{k-1,d}$
que l'on normalise en un endomorphisme $B_{k,d}(\uP)$ de $\Mmac_{k-1,d}$ via
$$
B_{k,d}(\uP) = \beta_{k,d}(\uP) \circ \beta_{k,d}(\uX^D)^{-1}
$$
Le miracle en question conduit à des liens très étroits entre trois
familles de scalaires toutes pilotées par $\minDiv$: les déterminants
de Cayley $\det B_{k,d}(\uP)$, les $\Delta_{k,d}(\uP)$ et \og nos\fg{}
déterminants excédentaires $\det W_{h,d}(\uP)$.  Les deux premières
familles sont reliées par l'égalité\footnote{%
Il faut ajouter l'information $\Delta_{n+1,d}(\uP) = 1$.  
Ainsi, en première approximation, on pourrait convenir que
l'égalité~$(\Delta_{k,d})$ est une définition de $\Delta_{k,d}(\uP)$
à l'aide de $\det B_{k,d}(\uP)$ 
par récurrence descendante sur~$k$. Encore faut-il montrer
que~$\Delta_{k+1,d}(\uP)$ divise $\det B_{k,d}(\uP)$! C'est la notion
de \emph{structure multiplicative}, abordée ultérieurement dans cette
introduction, qui permet de définir convenablement ces $\Delta_{k,d}(\uP)$.
}
$$
\det B_{k,d}(\uP) = \Delta_{k,d}(\uP)\,\Delta_{k+1,d}(\uP)
\leqno (\Delta_{k,d})
$$
Cette relation \og s'inverse\fg{} et permet d'exprimer
$\Delta_{1,d}(\uP)$ comme un quotient alterné des $\det B_{k,d}(\uP)$.
Ceci reste assez banal. Ce qui l'est moins, c'est d'une part que
$\Delta_{1,d}(\uP)$ est \emph {intrinsèque}, indépendant de tout
mécanisme de sélection (ce qui n'est pas le cas des
$\Delta_{k,d}(\uP)$ avec $k \ge 2$).  Et d'autre part que l'on a
l'égalité $\Delta_{1,d}(\uP) =\Res(\uP)$ pour $d\ge \delta+1$, égalité
permettant d'exprimer le résultant comme un quotient alterné de
mineurs issus des différentielles~$\partial_{k,d}(\uP)$.

Le caractère innovant réside dans la découverte d'autres relations,
cette fois absolument non banales, que nous avons qualifiées de \og
relations binomiales\fg{}, avec comme conséquence, deux égalités
fondamentales:
$$
\forall\ d : \qquad
\Delta_{2,d}(\uP) = \det W_{2,d}(\uP), \qquad\qquad
\det W_{1,d}(\uP) = \det W_{2,d}(\uP) \Delta_{1,d}(\uP)
$$
Les relations obtenues coiffent largement la formule de
Macaulay, cette dernière ne devenant qu'un cas particulier de ces
relations binomiales. Ces dernières permettent, entre
autres, d'exprimer le résultant comme un quotient de deux mineurs
issus de la \emph {première} différentielle.

\smallskip

Plus généralement, pour un degré~$d$ quelconque, notons $\chi_d$ la
caractéristique d'Euler-Poincaré du complexe $\rmK_{\sbullet,d}(\uX^D)$, égale
\footnote{%
Nous utiliserons parfois la notation $r_{0,d} = r_{0,d}(D)$ comme
synonyme de $\chi_d=\chi_d(D)$. Ceci est lié au fait que $\chi_d$ est
égal à $r_{0,d}$, cardinal de la base monomiale $\calS_{0,d}$ du
supplémentaire monomial $\Smac_{0,d}$ de $\Jex_{1,d}$.  Ces choix
terminologiques se justifieront dans la suite de l'étude.
}
à:
$$
\chi_d = \dim \bfA[\uX]_d / \Jex_{1,d} =
\# \{ \alpha\in \bbN^n \mid  |\alpha| = d \text{ et } \forall i\ \alpha_i<d_i\}
$$
Entre en jeu une forme $\chi_d$-linéaire alternée sur laquelle nous
reviendrons dans notre approche (voir en page~\pageref{IntroAproposDetd}, le point
\ref{IntroAproposDetd}) en particulier pour corriger un petit
mensonge:
$$
\Det_{d,\uP} : \BW^{\chi_d}(\bfA[\uX]_d) \to \bfA
\qquad
\begin {array}{c}
\text{qui passe au quotient} \\
\text{sans changer de nom} \\
\end {array}
\qquad
\Det_{d,\uP} : \BW^{\chi_d}(\bfB_d) \to \bfA
$$
Pour $d=\delta$, on a $\chi_\delta = 1$ et la forme
$\Det_{\delta,\uP}$ est la forme linéaire fondamentale $\omegaRes{\uP}
: \bfA[\uX]_\delta\to\bfA$.  Pour $d \ge \delta+1$, on a $\chi_d = 0$
et $\Det_{d,\uP}$ s'identifie au scalaire~$\Delta_{1,d}(\uP)$, donc
à~$\Res(\uP)$.  Pour exprimer cette forme $\chi_d$-linéaire alternée
$\Det_{d,\uP}$ sur~$\bfA[\uX]_d$ attachée à~$\uP$, on dispose d'un
quotient analogue à celui figurant dans la formule $(\deg = \delta)$
de la page précédente.

Références: chapitres~\ref{ComplexeDecompose}, \ref{ChapBW}.

\medskip

Terminons cette section en signalant le fait suivant.  Dans
l'article \cite{J7} de Jean-Pierre Jouanolou, on voit apparaître, en
section 3.10 (Formes de Sylvester et applications), pour $0 \le \nu
< \min(D)$, un scalaire~$c_\nu(\uP)$, très exactement dans le point
(c) de la proposition 3.10.14: ``Il existe un élément
$c_\nu(\uP) \in \bfA$ et un seul tel que~etc.''. Dans cet article, ce
scalaire, qui intervient comme dénominateur dans une expression de
$\Res(\uP)$, n'est pas explicité par une formule algébrique
directe. Il intervient dans notre exposé en degré $d$ où $(\nu,d)$ sont
complémentaires au degré critique~$\delta$. Avec nos notations, il s'agit~de:
$$
c_\nu(\uP) = \det W_{2,d}(\uP) \overset{\rm aussi}{=} \Delta_{2,d}(\uP)
\qquad \text{où} \qquad d=\delta-\nu
$$
Insistons sur le fait que, dans notre étude, le scalaire $\det
W_{2,d}(\uP)$ est défini pour tout $d$, indépendamment de~$\Res(\uP)$,
et que $W_{2,d}(\uP)$ est un endomorphisme explicite de $\Jex_{2,d}$,
qui plus est, simple à déterminer. Quant à~$\Delta_{2,d}(\uP)$, c'est une expression
de $\det W_{2,d}(\uP)$, expression liée à la
décomposition de Macaulay $\rmK_{\sbullet,d}(\uP)
= \Mmac_{\sbullet,d} \oplus \Smac_{\sbullet,d}$,
et se présentant comme un quotient alterné des $\big(\det B_{k,d}(\uP)\big)_{k\ge 2}$.
Tous ces objets, qui
sont pilotés par $\minDiv$, possèdent une $\sigma$-version. Enfin, les
formes de Sylvester dont il est question dans Jouanolou sont l'objet
du chapitre~\ref{ChapSylvesterHybride} et font intervenir la forme
$\chi_d$-alternée $\Det_{d,\uP}$.


\bigskip
Passons maintenant à un deuxième volet: le \emph{fondement} du résultant.

\subsection*{Le fondement du résultant}
\addcontentsline{toc}{subsection}{Le fondement du résultant}

A un moment donné de notre étude, difficile à
situer dans le temps, nous avons souhaité comprendre de manière approfondie
les notions sur lesquelles était assise la théorie du
résultant. Tâche complexe car la théorie du résultant ne représente
qu'un aspect particulier de la vaste théorie de l'élimination qui
remonte à plusieurs siècles.  A cette occasion, la lecture, partielle
et ardue, de l'\oe uvre de Jean-Pierre Jouanolou nous a été très utile.
Et de fil en aiguille, nous est venue l'idée, a priori saugrenue, d'un projet
ambitieux:
\og reprendre\fg{} les bases algébriques permettant de fonder le résultant
à l'aide de méthodes se voulant d'une part les plus effectives possible
et d'autre part les moins sophistiquées possible.

\bigskip

Bien entendu, la section précédente n'y contribue absolument pas
puisqu'elle suppose que bon nombre d'objets sont déjà en place, dont le
résultant.  Qu'est-ce que le résultant $\Res(\uP)$ de $n$ polynômes
homogènes $\uP = (P_1, \dots, P_n)$ en $n$ variables $\uX =
(X_1, \dots, X_n)$?  Quelle définition en prendre?  Quelle
caractérisation? Nous avons attendu 2018 pour choisir une définition
que nous jugeons structurelle et opérationnelle.
Trouver une architecture et une organisation cohérente s'est révélé relativement complexe.

\smallskip

Rappelons brièvement le contexte géométrique classique attaché au
théorème de l'élimination.

\medskip


\subsubsection*{Le théorème de l'élimination (approche \og standard\fg)}

Ce dernier peut se décliner sous plusieurs formes.  La lectrice pourra
par exemple consulter l'article de Demazure en \cite[Section 1,
proposition~2 intitulée ``théorème de l'élimination'']{Demazure2}.
Nous empruntons à Demazure la scholie (après la définition~2 de sa
section~3) qu'il tire de cette proposition~2.

\smallskip
Soit $\uP$ un système $(P_1, \dots, P_r)$ de $r$ polynômes
\emph {homogènes} à $n$ variables $\uX = (X_1, \dots, X_n)$, à coefficients dans un anneau~$\bfA$.
Définissons l'idéal saturé $\uPsat$ de $\langle\uP\rangle$ comme suit:
$$
\uPsat = \big\{F \in \bfA[\uX] \mid \exists e \ \forall i \in \{1,\dots,n\}
\ X_i^e F \in \langle\uP\rangle \big\}
$$
L'idéal d'élimination $\ElimIdeal$ possède la propriété suivante.
Pour tout morphisme $\rho : \bfA \to \bfk$ à valeurs dans un corps
$\bfk$, il y a équivalence:
$$
\rho\big(\ElimIdeal\big) = 0  \qquad \iff\qquad
\begin {minipage}{0.6\linewidth}
les ${^\rho}P_i$ possèdent un zéro commun non trivial au-dessus
d'une extension de $\bfk$, extension que l'on peut supposer de degré fini.
\end {minipage}
\leqno (\star)
$$
Le lecteur pourra également consulter le court
article~\cite{CartierTate} ``A simple proof of the main theorem of
elimination theory in algebraic geometry'' de Cartier \& Tate.

\bigskip

Ce qui nous concerne ici est le cas $r=n$, dit cas principal de
l'élimination. Dans ce contexte, le théorème principal de
l'élimination affirme qu'en \emph {terrain générique}, l'idéal
d'élimination est monogène.  Notons~$\uP^\gen$ le système générique de
format $D = (d_1, \dots, d_n)$.  Il y a un moyen de spécifier un
générateur précis $\Res(\uP^\gen) \in \bfA^\gen$ de l'idéal
d'élimination: c'est celui qui se spécialise en 1 lorsque l'on
spécialise~$\uP^\gen$ en le jeu étalon $\uX^D = (X_1^{d_1}, \dots,
X_n^{d_n})$.  Pour n'importe quel système~$\uP$ de format~$D$ à
coefficients dans~$\bfA$, on peut alors définir $\Res(\uP) \in \bfA$
comme étant la spécialisation de~$\Res(\uP^\gen)$ correspondant
à celle de~$\uP^\gen$ en $\uP$.  Ainsi par définition:
$$
\Res(X_1^{d_1}, \dots, X_n^{d_n}) = 1
$$
Il y a deux exemples élémentaires pour lesquels il est aisé d'exprimer
le résultant. Il y a celui des \og petites classes\fg{}: le résultant
de 2 polynômes en une variable\footnote{%
Il convient d'homogénéiser chacun des 2 polynômes de façon à être
en adéquation avec le contexte retenu ici.}
donné par le déterminant de la matrice de Sylvester.  Il y a
également celui d'un système $\uL$ de $n$ polynômes linéaires en $n$
indéterminées. On a alors $\Res(\uL) = \det(\dsL)$ où $\dsL$ est la
matrice $n \times n$ qui exprime les $L_j$ sur $(X_1, \dots,
X_n)$. Bien qu'élémentaires, montrer, sur ces deux exemples, qu'en
terrain générique le résultant est un générateur de l'idéal
d'élimination nécessite un certain effort.

\medskip

Par définition également, on a $\Res(\uP) \in \uPsat \cap\bfA$ mais en
général $\Res(\uP)$ n'est pas un générateur de l'idéal
d'élimination. On a cependant l'inclusion ci-dessous à droite:
$$
\Res(\uP)\bfA\  \subset\ \uPsat\cap\bfA\ \subset\  \sqrt{\Res(\uP)\bfA}
$$
Par exemple, en deux variables, on a $\langle aX,aY\rangle^\sat =
\langle a\rangle$ et 
$\Res(aX,aY)=a^2$, qui, sauf cas très particuliers,
n'engendre pas l'idéal d'élimination.

\medskip

Au dessus d'un corps $\bfk$, on peut reformuler le résultat~$(\star)$.
Pour un système $\uP = (P_1, \dots, P_n)$ de polynômes homogènes à
$n$ variables $\uX$ (autant de variables que d'indéterminées), à
coefficients dans un corps $\bfk$, il y a équivalence:
$$
\Res(\uP) = 0  \qquad \iff\qquad
\begin {minipage}{0.6\linewidth}
les $P_i$ possèdent un zéro commun non trivial au-dessus
d'une extension de $\bfk$, extension que l'on peut supposer de degré fini.
\end {minipage}
$$

\bigskip

Montrer en terrain générique la monogénéité de l'idéal d'élimination
demande un travail notable (voir par exemple
l'article \cite{Demazure2} de Demazure) et doit absolument être accompagné de
propriétés permettant de cerner le résultant telles que son degré
d'homogénéité en les $P_i$, le théorème de multiplicativité
etc.\footnote{Jean-Pierre Serre, dans la rédaction ``Résultant,
Discriminant'' pour Bourbaki, n$^\circ$ 627, juin 1976, dit lui-même
en commentaires, page 19, \og L'expérience prouve qu'il ne sert à rien
de connaître le résultant si on ne possède pas suffisamment de règles
de calcul, du genre de celles données dans les propositions etc.\fg.}

\medskip

Cette manière de procéder ne nous a pas convenu pour au moins deux
raisons.  D'une part, l'obtention de la monogénéité, qui se réalise en
terrain générique, ne nous semble pas très effective.  D'autre part,
comme nous l'avons déjà signalé, cette propriété de monogénéité ne se
spécialise pas.  Nous avons opté pour une toute autre stratégie, mise
en place au fur et à mesure de notre étude: elle permet d'associer à
une suite~$\uP$ régulière, un certain nombre d'ingrédients, dont le
résultant, avec possibilité de spécialisation.  Il n'y a plus de
géométrie mais des théories algébriques incluant l'algèbre
commutative, l'algèbre homologique et la combinatoire.  La section qui
vient explicite une partie de cette stratégie au caractère a priori
anti-intuitif.


\subsubsection*{Notre approche: résolutions libres finies et structure multiplicative}
\label {IntroGrandeOrganisatrice}

Nous utilisons la théorie des résolutions libres finies en
l'appliquant aux complexes libres que sont les composantes homogènes
$\rmK_{\sbullet,d}(\uP)$ du complexe de Koszul de $\uP$, pour
n'importe quel $d$, sous la seule clause ``$\uP$ régulière''.  Note:
l'utilisation du complexe $\rmK_{\sbullet,d}(\uP)$ n'est pas nouvelle
\footnote{%
Chez les deux auteurs cités, l'anneau de base $\bfA$ est supposé noethérien
et intégralement clos (voire factoriel chez Chardin), de manière à pouvoir
appliquer la théorie des diviseurs, en se référant au chapitre VII (Diviseurs)
d'Algèbre Commutative de Bourbaki. De notre côté, il n'y a pas
d'hypothèse particulière sur l'anneau de base $\bfA$, nous faisons porter les contraintes
sur le système $\uP$ et nous opérons avec les décompositions de Macaulay
de $\rmK_{\sbullet,d}(\uP)$ indexées par $\sigma\in\fS_n$. Cela s'est révélé commode
de dégager, dans la première section du chapitre \ref{ComplexeDecompose}, la
notion de complexe décomposé.}
cf. par exemple Chardin dans~\cite{Chardin2} ou Demazure
dans~\cite{Demazure1}.  Mais qu'en est-il de l'utilisation de
la \emph{structure multiplicative}?
Faut-il penser, dans ce
domaine, que cet outil n'a pas retenu l'attention des spécialistes?

Nous énumérons ci-dessous quelques points qui nous semblent
essentiels. Précisons que, dans notre étude, nous avons tenu à exposer
d'une part les rudiments de la théorie de la profondeur
(cf. le chapitre~\ref{ChapAC1}) et d'autre part un certain nombre de résultats
fondamentaux de la théorie des résolutions libres finies: c'est
l'objet du chapitre~\ref{ChapStructureMultiplicative}.

\begin {enumerate}[a)]
\item
\label {IntroStructureMultiplicative}
Exploitation de la structure multiplicative.  Rappelons de quoi il est
question pour un complexe exact libre $(F_\sbullet, u_\sbullet)$ où
$F_\sbullet = (F_k)_{0\le k \le n+1}$ avec $F_{n+1} = 0$, $u_\sbullet
= (u_k)_{1\le k \le n+1}$ et $u_k : F_k\to F_{k-1}$.  Nous utilisons
les notations du chapitre~\ref{ChapStructureMultiplicative} pour les
rangs attendus $(r_k)_{0 \le k \le n+1}$ définis de la manière
suivante:
$$
r_{n+1} = 0, \qquad r_{k+1} + r_k = \dim F_k
$$
On dit alors que $r_k$ est le rang
attendu de la différentielle $u_k$ pour $1 \le k \le n+1$.
De plus $r_0 = \chi$, caractéristique d'Euler-Poincaré de~$F_\sbullet$.

La terminologie ``structure multiplicative'' est liée à l'existence
d'une factorisation~$\BW^{r_1}(u_1) = \Theta\,\nu$, unique à un
inversible près, incluant des précisions sur la profondeur des
contenus de $\nu$ et $\Theta$:
$$
\vcenter{
\xymatrix @R=0.5cm{
\BW^{r_1}(F_1) \ar[dr]_{\nu}\ar[rr]^{\BW^{r_1}(u_1)} && \BW^{r_1}(F_0) \\
             & \bfA\ar[ur]_{\times \Theta} \\
}}
\qquad\qquad
\begin {array} {l}
\Gr(\rmc(\nu)) \overset{\rm def.}{=} \Gr(\Im\nu)\ge 2 \\[2mm]
\Theta \in \BW^{r_1}(F_0) \text{ sans torsion} \\
\end {array}
$$
Les initiateurs en sont Buchsbaum \& Eisenbud en~\cite{BE2}, Eagon \&
Northcott en~\cite {EN}.  Northcott a complété l'étude dans son
ouvrage~\cite{NorthcottFFR} et une approche élémentaire et constructive a été
accomplie par Coquand~\& Quitté en~\cite{CoquandQuitte}.

\smallskip

Il est essentiel de faire intervenir le module $M$ résolu par le
complexe, i.e. le conoyau $M := \Coker(u_1)$ de la première
différentielle et de dégager les objets intrinsèques à~$M$.  Par
exemple, l'entier $\chi$ ne dépend que du module~$M$; en effet, ce module
$M$ est de \emph{rang}~$\chi$ au sens où il existe un système (fini)
fidèle de localisations $\us = (s_\ell)$ de scalaires de~$\bfA$ tel
qu'en localisant en chaque $s_\ell$, le $\bfA[1/{s_\ell}]$-module
$M[1/s_\ell]$ est libre de rang $\chi$.  Fidèle signifie que
l'annulateur de $\us$ est réduit à 0, ce que l'on écrit~$\Ann(\us) = 0$.

\medskip

Grâce au fait que $(r_1,\chi)$ sont complémentaires à $\dim F_0$ i.e.
$r_1+\chi=\dim F_0$, il est préférable de \og remplacer\fg{} $\Theta$
par la forme $\chi$-linéaire alternée $\Theta^\sharp$ sur
$F_0$:
$$
\Theta^\sharp : \BW^\chi(F_0) \to \bfA,
\qquad
\Theta^\sharp = [\Theta\wedge\sbullet]_{\bfe_0}
\qquad
\text{où $\bfe_0$ est une orientation de $F_0$}
$$
La forme $\Theta^\sharp$ a le bon goût de passer au quotient pour
fournir une forme linéaire $\BW^\chi(M) \to \bfA$.

\medskip

Ce que l'on peut retenir: tout complexe exact libre $F_\sbullet$ de caractéristique
d'Euler-Poincaré $\chi$, de module résolu~$M$, donne naissance, à un
inversible près, à une forme $\chi$-linéaire alternée sur $M$ à
valeurs dans~$\bfA$. Note: on a le loisir de voir cette forme comme
une forme $\chi$-linéaire alternée sur le premier terme~$F_0$ du
complexe; on parle alors du déterminant de Cayley ``vectoriel'' du
complexe exact~$F_\sbullet$, cf. la section \ref{subsectionCayleyMacRae}.

\item
Pour nous, la théorie des résolutions libres ne peut pas être
dissociée de la théorie de la profondeur. A la base, il y a bien sûr
le fait que l'inégalité $\Gr(\ua) \ge k$, pour une suite~$\ua$ de
scalaires, peut être définie à l'aide du complexe de Koszul montant de~$\ua$
de la manière suivante:
$$
\text{$\Gr(\ua) \ge k$ signifie
que les $k$ premiers groupes de cohomologie 
$\rmH^0(\ua), \rmH^1(\ua),\dots, \rmH^{k-1}(\ua)$ sont nuls}
$$
D'autre part, l'exactitude de tout complexe libre est gouvernée par le
fameux théorème ``What makes a complex exact?''. Celui-ci fait
intervenir la profondeur des idéaux déterminantiels des différentielles
(ceux de rangs attendus).

\item
Le cas particulier $\Gr(\ua) \ge 2$, central dans notre étude,
possède la caractérisation élémentaire suivante.
Pour toute suite $\ub$ de scalaires de même longueur que $\ua$
vérifiant:
$$
a_i b_j = a_j b_i  \quad \forall i,\forall j
$$
il existe un seul scalaire $q$ tel que $b_i = qa_i$ pour tout $i$.

\medskip

Ceci nous a permis de définir une notion robuste de pgcd à savoir la
notion de \og pgcd fort\fg: un scalaire régulier $g$ est le pgcd fort
d'une suite $\ua$ si $g \mid a_i$ pour chaque $a_i$ et si
$\Gr(\ua/g) \ge 2$. On peut également définir cette notion au niveau d'un
module $M$: un vecteur sans torsion $w \in M$ est le pgcd fort d'une famille finie $\uv$
de $M$ si l'on a $v_i = q_iw$ avec $\Gr(\uq) \ge 2$.

\item
\label {IntroAproposDetd}
Notons $\bfB = \bfA[\uX]/\langle\uP\rangle$ et $\bfB_d$ sa composante
homogène de degré $d$.  En appliquant au
complexe~$\rmK_{\sbullet,d}(\uP)$ le point~a), 
on obtient via un mécanisme de normalisation, en degré $d=\delta$, une forme linéaire précise
$\omegaRes{\uP} \in (\bfB_\delta)^\star$, et, en degré
$d \ge \delta+1$, un scalaire~$\calR_d(\uP)$ qui n'est autre que le
résultant de $\uP$.

De manière générale, quelque soit $d$, on obtient une forme
$\chi_d$-alternée \emph{intrinsèque} \mbox{$\Det_{d,\uP}:\BdSod\to\bfA$} où
$\calS_{0,d}$ est la base monomiale de $\Smac_{0,d}$, supplémentaire monomial de $\Jex_{1,d}$:
$$
\calS_{0,d} =
\{X^\alpha  \mid  \alpha \in \bbN^n,\  |\alpha| = d,\ \forall i\ \alpha_i<d_i\}
$$
Quelques explications nous semblent nécessaires en rappelant que
$\chi_d = r_{0,d} \overset{\rm def.}{=}\#\calS_{0,d}$.
Définir une forme linéaire sur $\BW^{\chi_d}(\bfB_d)$ est équivalent à
définir une forme $\chi_d$-linéaire alternée sur
$\bfB_d^{\chi_d} = \bfB_d^{r_{0,d}}$. Certes $\bfB_d^{\chi_d} \simeq
\BdSod$, mais non de manière canonique. Pour parvenir à nos fins en vue
d'obtenir une forme $\Det_d$ intrinsèque, nous avons dû considérer $\BdSod$
\og comme le vrai visage\fg{} de $\bfB_d^{\chi_d}$.
Ce faisant, on retrouve~$\Delta_{1,d}$ via
$$
\Delta_{1,d} = \Det_d(\iota)
\qquad
\begin {array}{c}
\text{où $\iota$ est la composée de l'injection}\\
\text{et de la projection canoniques}\\
\end {array}
\qquad
\calS_{0,d} \hookrightarrow \bfA[\uX]_d \twoheadrightarrow \bfB_d
$$
Références: section~\ref{CanonicalCayleyDetSection}.

\item
La considération des sous-modules de Macaulay (pilotés par $\minDiv$) fournit
un moyen de calcul de $\Det_{d,\uP}$ de la manière suivante.
Il y a un constructeur explicite (piloté par $\minDiv$):
$$
\bsOmega_{d,\uP} : \bfB_d^{\calS_{0,d}} \to \End_\bfA(\bfA[\uX]_d)
$$
ainsi qu'une famille analogue, indexée par $\sigma\in\fS_n$, obtenue
en tordant par $\sigma$ les sous-modules de Macaulay et $\bsOmega_{d,\uP}$.
On dispose alors des formules:
$$
\Det_{d,\uP}(\sbullet) = \frac{\det\bsOmega_{d,\uP}(\sbullet)}{\det W_{2,d}(\uP)} =
\frac{\det\bsOmega^\sigma_{d,\uP}(\sbullet)}{\det W^\sigma_{2,d}(\uP)}
\qquad \forall\ \sigma \in \fS_n
$$
Il convient de noter que l'apparition du dénominateur
$\det W_{2,d}(\uP)$ est conséquence des \og relations binomiales\fg{}
évoquées dans la section~page~\pageref{IntroBinomialBW}.

\item
Nouvelle cerise sur le gâteau. Non seulement, nous n'avons pas à supposer $\uP$
générique, mais cette manière de procéder permet de répondre à la
question: ``Lorsque $\uP$ est régulière, à quelle condition le
résultant engendre-t-il l'idéal d'élimination?''.  La réponse est la
suivante: si et seulement si $\Gr(\omegaRes{\uP}) \ge 2$; de plus,
comme cette inégalité de profondeur est vérifiée lorsque $\uP$ est
générique, cela fournit une preuve de la monogénéité de l'idéal
d'élimination dans ce cas.
Références: théorème~\ref{4pointsPsuperreguliere},
lire également le commentaire après la proposition~\ref{WiebeCech}
ainsi que la section~\ref{HCech0B}.

\item
Etant donné un module librement résoluble~$M$, nous souhaitons lui
associer une notion ``d'invariant de MacRae''  en utilisant seulement la première
différentielle.  Ceci nous a conduit à revisiter cette notion.
Notre définition de module de MacRae s'éloigne de la
définition classique car elle ne se limite pas aux modules
de présentation finie de rang~0.  Elle est calquée sur celle des
modules librement résolubles en s'appuyant sur la notion de
pgcd fort.  Dans notre étude, l'invariant de MacRae n'est plus réduit
à un scalaire: nos modules de MacRae sont des modules~$M$ de
présentation finie, de rang constant $c$, portant de ce fait un \og
sous-module de Fitting\fg{} $\FittVect(M) \subset \BW^c(M)^\star$, à
qui nous demandons de posséder un pgcd fort, forme linéaire de
$\BW^c(M)^\star$ définie à un inversible près.
Références: chapitres~\ref{ChapFittingVectoriel} et \ref{ChapStructureMultiplicative}.
\end {enumerate}

\subsubsection*{Plusieurs visions du résultant, en privilégiant cependant
             $\Res(\protect\uP)=\omegaRes{\protect\uP}(\nabla)$}
\label {IntroDefPrivilegieeResultant}

On précise ici comment on va définir le résultant de $\uP$ en tant
qu'invariant de MacRae de certains $\bfA$-modules librement résolubles
de rang $0$ parmi les composantes homogènes des $\bfA[\uX]$-modules
gradués suivants:
$$
\bfB = \bfA[\uX] / \langle\uP\rangle
\qquad\qquad
\bfB' := \bfA[\uX] / \langle\nabla,\uP\rangle
$$
Par définition, $\nabla$ est \emph{un} déterminant bezoutien de $\uP$
i.e. $\nabla = \det(\dsV)$ où $\dsV \in \bbM_n(\bfA[\uX])$ est une
matrice \emph {homogène} réalisant $[P_1,\dots,P_n] =
[X_1,\dots, X_n]\dsV$, homogène signifiant que la colonne $j$ est
homogène de degré $d_j-1$. Ainsi $\nabla \in \bfA[\uX]_\delta$
et $X_i\nabla \in \langle\uP\rangle$ pour tout $i$, a fortiori $\nabla\in\uPsat_\delta$.

Deux déterminants bezoutiens sont égaux modulo
$\langle\uP\rangle_\delta$, cela résulte d'une propriété des suites
1-sécantes (cf. la proposition~\ref{IndependanceNabla}) appliquée à la
suite 1-sécante~$\uX$; on dispose ainsi d'un habitant bien précis
de~$\bfB_\delta$, que l'on note $\overline\nabla$, parfois $\nabla$
selon les jours de la semaine.

Malgré les apparences, il ne faudrait pas croire que nous allons nous
limiter au rang~$0$: le $\bfA$-module~$\bfB_\delta$, de rang 1, va jouer un rôle
capital. Désormais, on suppose $\uP$ régulière, ce qui fait que les
$(\bfB'_d)_{d\ge \delta}$ sont librement résolubles de rang 0 et
$\bfB_\delta$ est librement résoluble de rang 1.
Note: pour $d\ge\delta+1$, puisque $\nabla$ est homogène de degré $\delta$,
on a $\bfB'_d = \bfB_d$.

\medskip

Le résultat essentiel est que les $\bfA$-modules
$(\bfB'_d)_{d\ge \delta}$ ont même invariant de MacRae normalisé:
c'est ce scalaire commun qui sera déclaré être le résultant de $\uP$.
Voilà comment nous avons procédé au niveau des chapitres~\ref{ChapMacRaeForP}
et~\ref{ChapMacRaeDefResultant}:

\begin {enumerate} [a)]
\item
Normalisation et contrôle du poids en les $P_i$, objet du début du
chapitre \ref{ChapMacRaeForP}.  L'invariant de MacRae de~$\bfB_\delta$
est un sous-module libre de rang 1 de $(\bfB_\delta)^\star$ avec un
générateur privilégié $\omegaRes{\uP}$.  Il n'est pas difficile de
voir que le scalaire $\calR_\delta(\uP) := \omegaRes{\uP}(\nabla)$ est
un générateur de MacRae de
$\bfB'_\delta=\bfA[\uX]_\delta/\langle\nabla,\uP\rangle_\delta$.

\smallskip
De la même manière, pour $d \ge \delta+1$, l'idéal de MacRae de $\bfB'_d = \bfB_d$
possède un générateur privilégié noté~$\calR_d(\uP)$.

\item
Le début du chapitre \ref{ChapMacRaeDefResultant} montre l'égalité des $(\calR_d(\uP))_{d\ge\delta}$
en fournissant plusieurs preuves.
\end {enumerate}

Bilan: on prend comme définition
$$
\fbox{$\Res(\uP) \overset{\rm def.}{=} \calR_d(\uP)$ pour n'importe quel $d\ge\delta$}
$$
On dispose ainsi de plusieurs visions de $\Res(\uP)$: les scalaires
$\calR_d(\uP)$ attachés à $\bfB_d$ et le scalaire $\calR_\delta(P)
= \omegaRes{\uP}(\nabla)$ attaché à~$\bfB'_\delta$.

\smallskip
Pour corser le tout, au chapitre~\ref{ComplexeDecompose}, pour
$d\ge\delta+1$, interviendra un autre visage de $\calR_d(\uP)$, à
savoir $\calR_d(\uP) = \Delta_{1,d}(\uP)$. La glace est mince entre
ces deux scalaires: $\Delta_{1,d}(\uP)$, déterminant de Cayley
de~$\rmK_{\sbullet,d}(\uP)$ ou invariant de MacRae de $\bfB_d$, peut être
vu comme une expression de $\calR_d(\uP)$ en tant que quotient alterné
des $\big(\det B_{k,d}(\uP)\big)_{k\ge 1}$.
Références: la proposition \ref{InterpretationStructurelleDeltak}  et la
section \ref{SectionDeltakdMacaulay}, premier item du paragraphe ``Une
interprétation structurelle des~$\Delta_{k,d}(\uP)$'',
page \pageref{Delta1dMacRaeBd}.

\medskip

Insistons sur le fait que les points a) b) sont réalisés 
aux chapitres~\ref{ChapMacRaeForP}, \ref{ChapMacRaeDefResultant}
tandis que les $\Delta_{k,d}(\uP)$ n'interviennent
qu'à partir du chapitre~\ref{ComplexeDecompose} (on ne peut pas tout exposer d'un coup).

\medskip

{\footnotesize
En ce qui concerne b), des alternatives sont possibles en degré $d \ge \delta+1$ mais il est
impératif de veiller aux phénomènes oeuf/poule et à l'état d'avancement
de l'échafaudage. Au chapitre~\ref{ComplexeDecompose},
grâce aux décompositions de Macaulay tordues par $\sigma\in\fS_n$,
on définit les $\Delta^\sigma_{k,d}(\uP)$ pour n'importe quel $d$.
Pour $d\ge \delta+1$, 
on montre que $\Delta_{1,d}(\uP) = \Delta_{1,d+1}(\uP)$ par un argument
de profondeur~$\ge 2$ des~$\Delta^\sigma_{2,d}(\uP)$,
cf. le théorème~\ref{StabiliteDelta1d} et son corollaire.

Autre possibilité une fois acquise l'égalité $\Delta_{2,d}(\uP)=\det
W_{2,d}(\uP)$ donc $\det W_{1,d}(\uP) = \Delta_{1,d}(\uP)\det
W_{2,d}(\uP)$: utiliser l'égalité de stabilité de l'invariant de
MacRae en passant par l'égalité $\det W^\eq_{2,d+1}(\uP) = \det
W^\circ_{1,d+1}(\uP)$ pour $d \ge \delta+1$, cf. la combinatoire de la
section~\ref{sectionEgaliteStabiliteMacRae}.
}

\bigskip

Comme définition du résultant, on privilégie $\Res(\uP) =
\omegaRes{\uP}(\nabla)$, c'est-à-dire $\Res(\uP) = \calR_\delta(\uP)$,
pour plusieurs raisons. D'abord cette égalité
fait intervenir la forme linéaire $\omegaRes{\uP} : \bfA[\uX]_\delta
\to \bfA$, objet plus riche que son évaluation en~$\nabla$ et on va
ainsi bénéficier des propriétés de cette forme linéaire (en
particulier la propriété Cramer).  D'autre part, $\Res(\uP)$ et
$\omegaRes{\uP}$ sont étroitement liés via:
$$
\tau_0 :
\begin{array}[t]{rcl}
(\bfB_\delta)^\star & \longmapsto & \ElimIdeal \\[0.2cm]
\mu & \longmapsto & \mu(\nabla)
\end{array}
$$
Cette application linéaire $\tau_0$ est un isomorphisme lorsque $\uP$ est régulière; auquel cas
$\omegaRes{\uP}$ et $\Res(\uP)$ se correspondent dans cet isomorphisme.


\subsection*{Propriétés de la forme $\omegaRes{\protect\uP}:\bfA[\protect\uX]_\delta\to\bfA$}
\addcontentsline{toc}{subsection}{Propriétés de la forme $\omegaRes{\protect\uP}:\bfA[\protect\uX]_\delta\to\bfA$}


\noindent
\begin {quote}\raggedleft
{\sl La complexité est inhérente à l'existence puisqu'elle découle de la simplicité.}\par
Le peuple des connecteurs (2006), Thierry Crouzet.

\smallskip
{\sl Votre problème, c'est que vous envisagez la complexité comme le problème et non comme la solution.}\par
Les disparus (2007), Daniel Mendelsohn.

\smallskip
{\sl En supposant systématiquement $\uP$ générique, tout deviendrait plus simple. Mais moins intéressant.}\par
Le résultant (2023), Claude Quitté.
\end {quote}

\medskip

Plus sérieusement, il s'agit d'énoncer un certain nombre de propriétés
de~$\omegares$ (on se permet d'omettre~$\uP$ des notations).  La tâche
présente une certaine complexité à cause du statut du système $\uP$
(générique ou pas), de l'imbrication et inter-dépendance des
propriétés en question, du fait que certaines se spécialisent et
d'autres pas et enfin de la répartition des résultats sur plusieurs
chapitres.

Sont principalement concernés: les chapitres \ref{ChapMacRaeForP} (en
particulier les sections \ref{SectionomegaresProperties} et \ref{SectSystemesCreuxGr2})
et \ref{ChapElimination}. On y utilise les résultats du
chapitre \ref{ChapBgenerique}, qui lui-même utilise ceux du chapitre
\ref{ChapJeuCirculaire}.

\begin {enumerate} [a)]
\item
Propriété Cramer\footnote{%
Cette terminologie étrange provient du résultat suivant. Soient
$n-1$ vecteurs dans $\bfA^n$, $\uv = (v_2, \dots, v_n)$
et $\Delta_\uv$ la forme linéaire déterminantale sur $\bfA^n$ définie par
$\Delta_\uv := \Det(\sbullet,v_2, \dots, v_n)$.
Alors $\Delta_\uv(x)y - \Delta_\uv(y)x  \in \sum_i \bfA v_i$
pour $x,y \in \bfA^n$, 
conséquence directe de Cramer sur les systèmes linéaires carrés, cf.
la proposition~\ref{CramerSymetrie}.
}
de la forme $\omegares : \bfA[\uX]_\delta \to \bfA$
ou si l'on veut de $\uomegares : \bfB_\delta \to \bfA$.
Il s'agit de:
$$
\omegares(G)F - \omegares(F)G \in \langle\uP\rangle_\delta
\quad \forall\, F,G \in \langle\uP\rangle_\delta
\qquad\quad\text{ou encore}\qquad\quad
\uomegares(b)b' = \uomegares(b')b
\quad \forall\, b,b' \in \bfB_\delta
$$
C'est une propriété universelle i.e. vérifiée par tout système $\uP$,
et qui se spécialise. C'est l'objet du théorème~\ref{omegaresCramer}.

Impact direct: $X^\alpha \omegares(G)  \in \bfA G + \langle\uP\rangle_\delta$
pour $|\alpha| = \delta$ d'où l'implication
$$
G \in \uPsat_\delta \quad\Longrightarrow\quad \omegares(G) \in \ElimIdeal
$$
En particulier $\omegares(\nabla) \in \ElimIdeal$: la bonne nouvelle est que
le résultant est dans l'idéal d'élimination!

\medskip
De manière plus générale, pour un module quelconque $M$, on dit
qu'une forme linéaire $\mu : M \to \bfA$ a la propriété Cramer
si $\mu(m)m' = \mu(m')m$ pour tous $m,m' \in M$. Elle
a la propriété Cramer en $m_0 \in M$ si $\mu(m)m_0 = \mu(m_0)m$
pour tout $m\in M$.
Comme on a toujours $\mu(m)m' - \mu(m')m \in \Ker\mu$,
une forme injective a la propriété Cramer.

\item
Il y a deux propriétés de $\omegares$ que nous avons mis du temps à séparer:
$$
\text{injectivité de ${\uomegares}_{,\uP}$ \qquad versus\qquad $\Gr(\omegaRes{\uP}) \ge 2$}
$$
Sous le couvert de $\uP$ régulière, celle de gauche implique celle de droite,
point i) du théorème~\ref{omegaresGr2}.

Il n'y a pas de réciproque, cf la section~\ref{ExempleProfondeurNonInjective}.

\medskip

En générique, ${\uomegares}_{,\uP^\gen}$ est injective, cf le
théorème \ref{omegaresInjectiveDoncCramer}.  D'après ce qui précède
$\Gr(\omegaRes{\uP^\gen}) \ge 2$. On peut renforcer cette inégalité de
la manière suivante: l'image de $\omegaRes{\uP^\gen}$ contient une
suite régulière de longueur~2. De manière encore plus précise: le
point ii) de \ref{ResGenElimIdeal} affirme que pour tout $|\alpha|
= \delta$
$$
\text{la suite $\big(\omegaRes{\uP^\gen}(\nabla), \omegaRes{\uP^\gen}(X^\alpha)\big)$ est régulière}
$$
Nous avons fait figurer ce dernier résultat dans le
chapitre \ref{ChapElimination} car nous utilisons qu'en générique le
résultant engendre l'idéal d'élimination (cf le point suivant).

\smallskip

Par ailleurs, le théorème \ref{omegaresPgenGr2} fournit une autre
preuve de $\Gr(\omegaRes{\uP^\gen}) \ge 2$ sans utiliser l'injectivité
de~${\uomegares}_{,\uP^\gen}$.  Des résultats plus généraux concernant
$\Gr(\omegares) \ge 2$ figurent en section \ref{SectSystemesCreuxGr2},
section consacrée à l'étude de systèmes creux $\uP$ vérifiant
$\Gr(\omegaRes{\uP}) \ge 2$ et à des applications.

\item
En générique, le résultant engendre l'idéal d'élimination.
En fait, ce n'est pas la généricité qui gouverne mais $\Gr(\omegares) \ge 2$,
sous couvert de $\uP$ régulière. Pour comprendre cette phrase, voir le
théorème~\ref{4pointsPsuperreguliere}. Voir aussi le
théorème~\ref{ResGenElimIdeal}.

\item
La forme $\chi_d$-alternée $\Det_d$ possède la propriété Cramer, cf le
théorème \ref{DetdIsCramer}.

La propriété Cramer pour une forme $r$-linéaire alternée $\mu : M^r \to \bfA$
est la suivante, par exemple pour $r=3$:
$$
\mu(m_1,m_2,m_3)m = \mu(m,m_2,m_3)m_1 + \mu(m_1,m,m_3)m_2 + \mu(m_1,m_2,m)m_3
$$

\item
``Cramer'' versus ``Cramer en $\nabla$''.
Il ne faut surtout pas confondre ces deux notions.
\emph {Toute} forme $\mu \in (\bfB_\delta)^\star$ est de Cramer en $\nabla$ i.e.
$$
\mu(\nabla)b = \mu(b)\nabla \qquad \forall b \in \bfB_\delta
$$
C'est cette propriété qui entraîne $\mu(\nabla) \in \ElimIdeal$ pour $\mu \in (\bfB_\delta)^\star$
et légitime la définition de $\tau_0 : (\bfB_\delta)^\star \to \ElimIdeal$ par
$\mu\mapsto\mu(\nabla)$.
A cette occasion, le lecteur pourra consulter la section~\ref{HCech0B}.
Mais il existe des formes $\mu \in (\bfB_\delta)^\star$ ne vérifiant pas Cramer,
voir un exemple ci-dessous. Le caractère ``être de Cramer'', vérifié par $\uomegares$, n'est donc
pas banal.

On peut énoncer cette propriété en remontant au niveau
$\bfA[\uX]_\delta$: toute forme linéaire $\mu: \bfA[\uX]_\delta \to \bfA$
\emph {nulle} sur $\langle\uP\rangle_\delta$ vérifie
$\mu(\nabla)F - \mu(F)\nabla \in \langle\uP\rangle_\delta$ pour $F \in \bfA[\uX]_\delta$.

\medskip

Considérons par exemple le jeu étalon généralisé $\uP = (p_1X_1^{d_1},
\dots, p_nX_n^{d_n})$ où $p_1, \dots, p_n$ sont $n$ indéterminées sur
$\bbZ$, l'anneau de base $\bfA = \bbZ[p_1,\dots,p_n]$ et la forme
linéaire $\mu : \bfA[\uX]_\delta \to \bfA$ coordonnée sur le mouton
noir $\uX^\emouton$ (coordonnée relativement à la base monomiale de
$\bfA[\uX]_\delta$). Il est clair que~$\mu$ passe au quotient
modulo $\langle \uX^D\rangle_\delta$, a fortiori sur
$\bfB_\delta = \bfA[\uX]_\delta/\langle\uP\rangle_\delta$.
Et on a:
$$
\mu(X^\emouton)X^\alpha - \mu(X^\alpha)X^\emouton =
\begin {cases}
X^\alpha  &\text{si $X^\alpha \ne X^\emouton$} \\
0        &\text{si $X^\alpha = X^\emouton$} \\
\end {cases}
$$
Supposons $D \ne (1,\dots,1)$ de sorte qu'il existe un monôme
$X^\alpha$ de degré $\delta$ autre que $X^\emouton$.  Alors~$\mu$
n'est pas de Cramer car, dans la branche du haut,
$X^\alpha \notin \langle\uP\rangle_\delta$.  On évite cette dernière
négation à l'aide d'un énoncé plus précis où l'on remplace $\bfA
= \bbZ[p_1,\dots,p_n]$ par la donnée de $n$ éléments réguliers
$p_1, \dots, p_n$ d'un anneau $\bfA$. Alors, dans le cas où il y a
deux indices~$k$ tels que $d_k \ge 2$, la forme $\mu =
(X^\emouton)^\star$ est de Cramer si et seulement si chaque $p_i$ est
inversible; dans le cas d'un seul $k$ tel que $d_k \ge 2$, $\mu$ est
de Cramer et seulement si chaque $p_i$ pour $i\ne k$ est inversible.

\end {enumerate}

\subsection*{Jeu circulaire et variantes. Applications à la régularité
(concernés: \ref{ChapJeuCirculaire}, \ref{ChapBgenerique})}
\addcontentsline{toc}{subsection}{Jeu circulaire et variantes. Applications à la régularité}


On rappelle que $\bfB$ est le $\bfA[\uX]$-module gradué
$\bfA[\uX]/\langle\uP\rangle$ et que $\uP^\gen$ est le système
générique de format $D = (d_1, \dots, d_n)$. Un des objectifs est
d'exhiber certains scalaires qui sont réguliers sur~$\bfA[\uX]/\uPsat$
(ou bien sur $\bfB_\delta$, voire sur certains $\bfB_d$ ou sur $\bfB$
tout entier), du moins en terrain générique, et d'en montrer plus tard
l'impact sur l'injectivité de ${\uomegares}_{,\uP^\gen}
: \bfB_\delta\to\bfA$.

L'encadré {\bf But et conséquence} en deuxième page
du \ref{ChapJeuCirculaire} et la première page du \ref{ChapBgenerique}
contiennent un certain nombre d'explications sur les objectifs,
contenus et stratégies de ces chapitres.  Ici nous essayons
d'apporter, peut-être avec un certain recul, et sans vouloir nous
répéter, un complément d'informations.

\medskip

La forme fondamentale $\omegaRes{\uP} : \bfA[\uX]_\delta \to \bfA$
n'est pas encore en place aux
chapitres \ref{ChapJeuCirculaire}, \ref{ChapBgenerique}: elle devra
attendre le chapitre~\ref{ChapMacRaeForP}.
Si bien qu'en guise de substitut, nous utilisons la forme
linéaire déterminantale $\omega_\uP : \bfA[\uX]_\delta \to \bfA$
associée à l'application de Sylvester $\Syl_\delta(\uP) :
\rmK_{1,\delta} \to \bfA[\uX]_\delta$.  Cette forme $\omega_\uP$,
pilotée par le mécanisme de sélection $\minDiv$, a été étudiée au
chapitre précédent, en section \ref{sectionFormeLineaireOmega}. On
disposera plus tard de la relation de divisibilité $\omegaRes{\uP}
\mid \omega_\uP$: en fait, $\omegaRes{\uP}$ est le pgcd fort de toutes
les formes déterminantales $\bfA[\uX]_\delta \to \bfA$ associées
à~$\Syl_\delta(\uP)$, du moins lorsque $\uP$ est régulière.

On montrera (théorème \ref{omegaInjective}) que la forme
$\overline\omega_{\uP^\gen}$ est injective; par conséquent, il en sera
de même de~${\uomegares}_{,\uP^\gen}$ qui la divise.

\medskip

On se concentre désormais sur $\omega_\uP$ et son évaluation $\det
W_{1,\delta}(\uP) = \omega_\uP(X^\emouton)$ où $X^\emouton$ est le
mouton-noir $X^\emouton = X_1^{d_1-1} \cdots X_n^{d_n-1}$.

Nous allons exhiber un système $\uQ$, à coefficients
$0,\pm1$, de format $D$, tel que $\Un = (1,\dots,1)$ soit un zéro
commun des $Q_i$ et vérifie la condition capitale \boxed{\det
W_{1,\delta}(\uQ)=1}.  Il en résultera que $\det
W_{1,\delta}(\uP^\gen)$ est régulier sur
$\bfA[\uX]/\langle\uP^\gen\rangle$ (théorème \ref{GenPsatRegScalars}),
conséquence du résultat général~\ref{PsatReg}. On sera alors
en mesure de prouver l'injectivité de $\overline\omega_{\uP^\gen}$
ainsi que la régularité sur $\bfA[\uX]/\langle\uP^\gen\rangle$
de certains déterminants (une liste figure dans le théorème \ref{GenPsatRegScalars}).

\subsubsection*{Le jeu circulaire $\uQ$, retenu pour  vérifier $\det W_{1,\delta}(\uQ)=1$}

Notre choix initial s'est porté sur le système \emph {binomial} suivant, dit circulaire:
$$
Q_i = X_i^{d_i-1}(X_i - X_{i+1}) = X_i^{d_i} - X_i^{d_i-1} X_{i+1}
\qquad\qquad
\text{avec $X_{n+1} = X_1$}
$$
Nous qualifions de \emph {combinatoire binomiale en degré critique} la
panoplie d'outils qui va nous permettre de prouver que $\det
W_{1,\delta}(\uQ)=1$. Mieux: dans une base monomiale adéquate de
$\Jex_{1,\delta}$, l'endomorphisme~$W_{1,\delta}(\uQ)$ est
triangulaire à diagonale unité, voir par exemple l'illustration
en page~\pageref{NOTA06-wMiterateur}, à la fin
de la section~\ref{IntroJeuCirculaire}. 
Et le jeu $\uQ$ possède bien d'autres propriétés.

\bigskip

Une question se pose: d'où sort ce système binomial $\uQ$?
Initialement, nous l'avons deviné en cherchant un jeu binomial ayant
comme seul zéro projectif $\Un = (1:\cdots:1)$, cf. la première page
du chapitre~\ref{ChapJeuCirculaire} ainsi que la définition d'un jeu
simple en~\ref{DefJeuSimple}.  En notant $\uR = (R_1, \dots, R_n)$ le
système monomial de format $D$ défini par $R_i=X_i^{d_i-1} X_{i+1}$,
on peut déjà remarquer que $\uQ$ est de la forme $\uQ = \uX^D - \uR$
et que la définition pointée porte sur $\uR$ plutôt que sur $\uQ$.

\medskip

Au fur et à mesure, notre démarche va connaître des évolutions. D'une part,
toujours avec le jeu circulaire, nous allons génériser $\uX^D - \uR$
en $\bsp\uX^D + \bsq\uR = (p_iX_i^{d_i} + q_iX_i^{d_i-1}X_{i+1})_{1\le i \le n}$ où $p_i, q_i$
sont $2n$ indéterminées sur~$\bbZ$, de manière à \og mieux y voir\fg
\footnote {Pour ce système $\bsp\uX^D + \bsq\uR$, nous avons utilisé la terminologie
``jeu circulaire généralisé''. Peut-être que ``jeu circulaire générisé''
aurait été préférable?  
}
.  
D'autre part, nous allons montrer que $\omega_\uQ(X^\alpha) \overset{\rm def.}{=}
\det\Omega_\uQ(X^\alpha)$ est égal à 1 pour n'importe
quel $X^\alpha$ de degré $\delta$ et pas seulement pour $X^\alpha=X^\emouton$.

Notre travail va consister à exhiber des bases monomiales adéquates de
$\Jex_{1,\delta}$, ou bien de $\bfA[\uX]_\delta$, de façon à contrôler
le déterminant des endomorphismes \og en degré $\delta$\fg:
$$
W_{1,\delta}(\uQ),\quad W_{1,\delta}(\bsp\uX^D + \bsq\uR)
\qquad\qquad
\Omega_\uQ(X^\alpha),\quad \Omega_{\bsp\uX^D + \bsq\uR}(X^\alpha)
\qquad
(\deg X^\alpha = \delta)
\leqno(\star)
$$
Cette tâche va être assurée en introduisant un itérateur monomial
$w_{1,\delta}(\uR)$ (défini en section~\ref{NOTA06-wMiterateur})
piloté par $\minDiv$ ainsi que son graphe, et en montrant un \og
théorème du puits\fg{} en \ref{TheoPuits} (des exemples précédant
l'énoncé en fournissent une illustration).  Il est remarquable que la lecture du
graphe de l'itérateur~$w_{1,\delta}(\uR)$ permette la détermination de
chaque $\omega_{\bsp\uX^D + \bsq\uR}(X^\alpha)$, voir l'exemple
\ref{Graphe122}.

Si on regarde de près la preuve de terminaison initialement prévue
pour $\minDiv$, on constate une propriété indépendante de tout
mécanisme de sélection, ce qui permet de prouver par exemple que $\det
W^\sigma_{1,\delta}(\uQ) = 1$ pour tout $\sigma\in \fS_n$.  Ceci est
dû au fait que la terminaison intervenant dans ce théorème du puits
est certifiée par une fonction hauteur $h$, définie sur les monômes de degré
$\delta$, à valeurs dans $\bbN$, ne dépendant que de $\uR$ et vérifiant:
$$
h(X^\emouton) = 0,
\qquad
h\big((X^\alpha/X_i^{d_i}) R_i\big) = h(X^\alpha)-1
\qquad
\begin {array}{c}
\text {pour tout $X^\alpha\ne X^\emouton$}
\\
\text{ et n'importe quel $i$ tel que $X_i^{d_i} \mid X^\alpha$}
\end {array}
\leqno(\heartsuit)
$$
Cette fonction hauteur $h$ est obtenue à partir d'une autre fonction $\calH_0
: \bbZ^n_0 \to \bbN$ (objet du théorème~\ref{FonctionHauteur}) où l'on
note $\bbZ^n_0$ le sous-module de $\bbZ^n$ constitué des $x$ tels que
$\sum_i x_i = 0$.  En désignant par $(\varepsilon_1, \dots,
\varepsilon_n)$ la base canonique de $\bbZ^n$ et en convenant de
$\varepsilon_{n+1} = \varepsilon_1$, cette fonction $\calH_0$ vérifie:
$$
\calH_0({\bf 0}) = 0,
\qquad
\calH_0\big(x - (\varepsilon_i - \varepsilon_{i+1})\big) =  \calH_0(x) - 1
\qquad
\begin {array}{c}
\text{pour tout $x\ne{\bf 0}$}
\\
\text{et n'importe quel $i$ tel que $x_i > 0$}
\end {array}
$$
La provenance de $\varepsilon_i - \varepsilon_{i+1}$ s'explique par l'égalité:
$$
\varepsilon_i - \varepsilon_{i+1} = \exposant(X_i^{d_i}) - \exposant(R_i)
$$
En réalisant la translation entre $\bbZ^n_0$ et $\bbZ^n_\delta$
définie par $\alpha = x + (\emouton) \leftrightarrow
x=\alpha-(\emouton)$, on peut faire de $\calH_0$ une fonction sur
$\bbZ^n_\delta$ puis la restreindre à $\bbN^n_\delta$, pour obtenir,
modulo la correspondance $\alpha \leftrightarrow X^\alpha$, la fonction~$h$.

\medskip

En guise de résumé: nous avons la main sur les endomorphismes en
$(\star)$, et sur bien d'autres, par exemple sur les versions tordues
par $\sigma\in\fS_n$.  Le contrat initial encadré $\det
W_{1,\delta}(\uQ) = 1$ est largement rempli.

\subsubsection*{Le glissement ou l'ouverture $\protect\uQ := \protect\uX^D-\protect\uR$
               où $\protect\uR$ est monomial de format $D$}


Nous nous sommes posés la question de  l'existence de \og bons\fg{}
systèmes monomiaux $\uR$ de format~$D$ tels que, de manière informelle,
$\uX^D-\uR$ possède, en degré $\delta$, des qualités analogues à celui
du jeu circulaire.  C'est l'objet de la dernière
section~\ref{SectionJeuxSimples}.

Au début de cette section~\ref{SectionJeuxSimples}, nous avons choisi,
comme formalisation de \og bon\fg, l'existence d'une fonction hauteur
$h : \bbN^n_\delta \to \bbN$ vérifiant, modulo la correspondance
$\alpha\leftrightarrow X^\alpha$, les contraintes $(\heartsuit)$.

Pourquoi ne pas demander à $\uR$ d'être simple au sens de la
définition~\ref{DefJeuSimple} i.e.  au système binomial homogène
$\uX^D - \uR$ de possèder $\Un = (1:\cdots:1)$ comme unique zéro
projectif? En fait, nous soupçonnons que ces deux notions (existence
d'une fonction hauteur sur $\bbN^n_\delta$ et simplicité) sont équivalentes
mais nous n'avons pas su le prouver, cf. l'énoncé
conjectural~\ref{ConjectureJeuSimple}. L'existence d'une fonction
hauteur pour $\uR$ nous semble un choix plus opérationnel et entraîne
aisément la simplicité de~$\uR$; c'est la réciproque qui est problématique.

\smallskip

Nous avons également attaché à $\uR$ (cf. page \pageref {GraphOfR}) le
graphe orienté dont les sommets sont les monômes de degré~$\delta$,
avec des arcs étiquetés: un arc $X^\alpha \overset{i}{\to} X^\beta$
pour chaque~$i$ vérifiant
$$
X_i^{d_i} \mid X^\alpha \quad\text{et}\quad X^\beta = \frac {X^\alpha}{X_i^{d_i}} R_i
\leqno (\text {règle de réécriture numéro $i$})
$$
L'existence d'une telle fonction hauteur $h$ (qui est unique
lorsqu'elle existe) est équivalente au fait que ce graphe est sans
circuit (une des équivalences du
théorème~\ref{EquivalencesExistenceHauteur}).  Dans ce graphe,
$h(X^\alpha)$ représente la distance entre $X^\alpha$ et $X^\emouton$
au sens de la longueur minimum d'un chemin (orienté) de $X^\alpha$ à
$X^\emouton$ (en fait, tous les chemins de $X^\alpha$ à $X^\emouton$
ont même longueur).

\bigskip

La matrice jacobienne $\Jac_\Un(\uX^D-\uR) \in \bbM_n(\bbZ)$ et ses
$n$ mineurs principaux notés $(r_1, \dots, r_n)$, qui sont toujours
$\ge 0$, vont jouer un rôle important, voir par exemple le
corollaire~\ref{CorollaireExistenceHauteur} et la
proposition~\ref{ExistenceSuperHauteur}. Cette dernière affirme que
les conditions $\pgcd(r_1, \dots, r_n) = 1$ et $r_i \ge 1$ pour tout
$i$ entraînent l'existence d'une fonction hauteur $h_0$ sur
$\bbZ^n_0$, a fortiori, via la translation par $\emouton$, celle d'une fonction hauteur $h$ sur
$\bbN^n_\delta$.  On retrouve ainsi les résultats correspondant au
jeu circulaire pour lequel $r_i=1$ pour tout $i$.


Notons $\uQ = \uX^D - \uR$.
La matrice jacobienne $\Jac_1(\uQ)$ se détermine facilement via
ses lignes $(\ell_1, \dots, \ell_n)$ puisque $\ell_i$
est définie par $\ell_i = \exposant(X_i^{d_i}) - \exposant(R_i)$.
Et comme $\deg(R_i)= d_i$, on a $\ell_i \in \bbZ^n_0$.
On a donc une matrice à coefficients entiers dont la somme des colonnes est nulle:
$$
\Jac_\Un(\uQ) =
\begin {bmatrix}
d_1    &0 & \cdots & 0  \\
0      &d_2 &      & 0  \\
\vdots &    &\ddots     \\
0      &    &0     &d_n \\
\end {bmatrix}
-
\begin {bmatrix}
\exposant(R_1) \\
\exposant(R_2) \\
\vdots \\
\exposant(R_n) \\
\end {bmatrix}
$$
Sa ligne $\ell_i$ vérifie $\ell_i(\uX) := Q_i(\uX + \Un)_1$ où l'indice 1 désigne
la composante homogène de degré 1.  En caractéristique~0, le
point~$\Un$ est toujours un zéro simple de la sous-variété projective $\{\uQ=0\}
\subset \bbP^{n-1}$ vu qu'en ce point le cône tangent, égal à
l'espace tangent, est défini par la nullité des $\ell_i$:
$$
\calC_\Un(\uQ = 0) = \calT_\Un(\uQ = 0)  = \{\ell_1 = \cdots = \ell_n = 0\} \subset \bbP^{n-1}
$$

\medskip

A ce niveau, il nous semble difficile de donner plus de détails et
nous invitons la lectrice à consulter cette
section~\ref{SectionJeuxSimples}.  L'étude qui y est réalisée va nous
permettre d'exhiber de bons systèmes monomiaux autres que le jeu
circulaire, indispensables pour la construction de systèmes creux
$\uP$ vérifiant l'inégalité $\Gr(\omegaRes{\uP}) \ge 2$ (cf. la section ultérieure
\ref{SectSystemesCreuxGr2}).

\bigskip

Pour conclure, signalons, concernant les systèmes binomiaux $\uX^D -
\uR$ où $\uR$ est de format $D$, que de nombreuses propriétés n'ont
pas été abordées dans notre étude et qu'un certain nombre de mystères
subsistent.  C'est le cas par exemple de l'adéquation
conjecturale~\ref{ConjectureJeuSimple} entre la simplicité de $\uR$ et
l'existence d'une fonction hauteur sur $\bbN^n_\delta$, résultat que
nous trouvons surprenant (voir l'exemple après l'énoncé~\ref{ConjectureJeuSimple}).

\subsection*{Autour du théorème de Wiebe}
\addcontentsline{toc}{subsection}{Autour du théorème de Wiebe}

Sont concernés: l'énoncé général \ref{WiebeTheorem},
sa preuve en annexe \ref{WiebeProof} ainsi que les applications
au contexte polynomial données en section \ref{SectCompHmgSature}.

\bigskip

Nous avons déjà fait intervenir $\nabla \in \bfA[\uX]_\delta$,
déterminant bezoutien de $\uP$, en
page \pageref{IntroDefPrivilegieeResultant}. Il vérifie $X_i\nabla \in \langle\uP\rangle$,
en particulier $\nabla\in\uPsat_\delta$. Lorsque $\uP$ est régulière, on dispose du
résultat plus précis suivant (c'est un point de la proposition~\ref{MiniWiebe}):
$$
\uPsat_\delta = \bfA\nabla \oplus \langle\uP\rangle_\delta
$$
En rappelant que $\bfB=\bfA[\uX]/\langle\uP\rangle$, il est d'usage de
noter $\vH^0_{\uX}(\bfB)_d = \uPsat_d/\langle\uP\rangle_d$ pour tout
$d$, en particulier $\vH^0_{\uX}(\bfB)_0 \overset{\rm def.}{=}
\ElimIdeal$, et d'introduire un morphisme noté $\tau_\delta$, à mettre
sur le même plan\footnote{%
La définition de $\tau_\delta$ repose sur
l'appartenance $\nabla\in\ElimIdeal$, appartenance résultant aisément
de la définition de $\nabla$. Par contre, celle de $\tau_0$ nécessite la propriété
``Cramer en $\nabla$'' des formes de $(\bfB_\delta)^\star$, ce qui n'est pas une
mince affaire.}
que $\tau_0$:
$$
\tau_0:
\begin {array}[t]{rcl}
(\bfB_\delta)^\star & \longrightarrow &
 \vH^0_{\uX}(\bfB)_0
\\[2mm]
\mu & \longmapsto & \mu(\overline\nabla)
\\
\end {array}
\qquad\qquad
\tau_\delta:
\begin {array}[t]{rcl}
(\bfB_0)^\star = \bfA & \longrightarrow &
\vH^0_{\uX}(\bfB)_\delta
\\[2mm]
1 & \longmapsto & \overline\nabla
\\
\end {array}
$$
La détermination ci-dessus de $\uPsat_\delta$ lorsque $\uP$ est
régulière peut être énoncée en disant que $\tau_\delta$ est un
isomorphisme.  Ce n'est pas uniquement pour faire savant que l'on
procède ainsi. On se doute en effet
que $\vH^0_{\uX}(\bfB)$ n'est pas qu'une notation: c'est un $\bfB$-module gradué,
qui est, en degré cohomologique~$0$, le ``groupe'' de cohomologie de {\Cech} de
la suite~$\uX$ sur $\bfB=\bfA[\uX]/\langle\uP\rangle$. Et
il y a un traitement \emph
{uniforme} dans lequel intervient un morphisme $\tau_d :
(\bfB_{\delta-d})^\star \to \vH^0_{\uX}(\bfB)_d$, composante homogène
de degré $d$ d'un $\bfB$-morphisme de transgression~$\tau$.
Nous invitons la lectrice à consulter la section~\ref{HCech0B};
nous avons tenté d'y réaliser 
une description de la transgression $\tau$ du bicomplexe $\rmK^{n-\sbullet}(\uP\,;\bfA[\uX])
  \otimes\vC_\uX^{n-\sbullet}(\bfA[\uX])$, sans toutefois rentrer dans tous les détails
et sans fournir les preuves, ce qui aurait
nécessité un trop grand investissement en cohomologie de {\Cech}.

\medskip

Au niveau du chapitre~\ref{ObjetsSuiteP}, nous avons choisi d'utiliser
le théorème de Wiebe dont l'énoncé général en~\ref{WiebeTheorem} est
élémentaire (il n'y a pas de cohomologie), en déportant la preuve en
annexe~\ref{WiebeProof}; dans cette annexe, nous avons opté pour une
preuve homologique du théorème de Wiebe.

\smallskip

L'énoncé général de Wiebe, au dessus d'un anneau commutatif
de base quelconque, fait intervenir:

\begin {enumerate}
\item  
deux suites scalaires $\ux$, $\ua$ de même longueur
complètement sécantes, vérifiant l'inclusion $\langle\ua\rangle \subseteq
\langle\ux\rangle$
\item
le déterminant de n'importe quelle matrice carrée certifiant cette inclusion
\end {enumerate}
Ce que nous appelons le
contexte polynomial est celui où l'anneau de base est $\bfA[\uX]$,
avec les deux suites~$\uX$ et $\uP$, la suite $\uX$ correspondant
à~$\ux$ et la suite $\uP$ à~$\ua$, le déterminant étant un déterminant
bezoutien~$\nabla$ de $\uP$.


\smallskip

Dans la section \ref{SectCompHmgSature}, à partir du théorème de Wiebe
appliqué à $(\uX,\uP)$, qui nécessite que $\uP$ soit régulière, nous en tirons
des conséquences sur les composantes homogènes de degré $d$; en
particulier pour $d\ge \delta$:
$$
\uPsat_\delta = \bfA\nabla \oplus \langle\uP\rangle_\delta,
\qquad\qquad
\uPsat_d = \langle\uP\rangle_\delta
\qquad  d \ge \delta+1
$$
On obtient également $\Ann(\bfB'_\delta) = \Ann(\bfB_d) = \ElimIdeal$ pour $d\ge\delta+1$
et certains renseignements en dessous du degré critique~$\delta$.

\subsection*{En guise d'épilogue}
\addcontentsline{toc}{subsection}{En guise d'épilogue}

Les pages précédentes constituent un aperçu partiel de notre étude.
Ajoutons deux points que nous jugeons importants (le premier
est utilisé à de nombreuses reprises):

\begin{enumerate}
\item
Contrôle de la profondeur par les composantes homogènes dominantes:
théorème \ref{ControleProfondeur}
\item
Contrôle de la profondeur par la propriété de couverture:
théorème \ref{GrViaPcouvreQ}
\end{enumerate}

Par ailleurs, nous avons tenu à traiter assez rapidement
(chapitre \ref{DeuxCasEcole}) deux cas d'école, d'une part le cas
$n=2$, et d'autre part celui du format $D = (1,\dots,1,d_n)$, pour
lesquels on dispose d'une détermination simple du résultant
et de la forme $\omegaRes{\uP}$. 

\medskip

Il ne nous a pas semblé opportun de réaliser une description de notre
travail chapitre par chapitre via une énumération fastidieuse. Nous
pensons que la table des matières en donne un avant-goût.  Nous avons
affublé trois chapitres généraux de l'acronyme AC$_i$ (Algèbre
Commutative).  Signalons que le chapitre~\ref{ChapSeries} (Séries de
Hilbert-Poincaré), qui vient en fin d'étude, nous a été très utile
pendant tout le développement de notre travail. Dès le départ (2012),
nous avons mis au point la détermination sous forme de fractions
rationnelles de certaines séries dimensionnelles ou de poids en $P_i$; cela
nous a permis de contrôler ou de montrer certains résultats et parfois
de les deviner.

\medskip

Nous avons eu besoin de nombreux exemples pour comprendre. Nous en
avons inclus certains. Nous reconnaissons avoir pris parfois beaucoup
de place en voulant traiter des thèmes qui nous tenaient à coeur même
s'ils contribuaient modérément à l'étude du résultant. Nous pensons par exemple à la
section \ref{SectomegaresJeuxSimplex} dans laquelle nous avons tenu à
étudier la forme $\omegares$ de systèmes $\bsp\uX^D + \bsq\uR$ où
$\uR$ est un jeu monomial simple.  Pourquoi inclure un tel traitement?
Tout simplement pour laisser des traces d'une étude incomplète concernant
la combinatoire binomiale de $\uX^D - \uR$ en degré critique.

\medskip

En guise de conclusion: nous nous sommes efforcés de rendre compte
d'une dizaine d'années de travail, et nous avons bien conscience qu'il
reste encore des choses à comprendre!

\medskip

\begin{quote}\raggedleft
{\sl Ecrire pour comprendre. Ne pas encaisser silencieux.}\par
Minimal minibomme (1984), Marc Gendron.

\smallskip
{\sl Le danger, ce n'est pas ce qu'on ignore, c'est ce que l'on tient pour certain et qui ne l'est pas.}\par
Mark Twain.
\end{quote}

\vskip 1cm

\centerline {Claude Quitté, Claire Tête, Juin 2023}
\pagenumbering{arabic} 

\section{AC$_1$: dualité en algèbre extérieure, profondeur, théorème de Wiebe}
\label{ChapAC1}

\subsection{Un peu d'algèbre linéaire : la formule de Cramer matricielle}

Pour une matrice carrée $A \in \bbM_n(\bfA)$ dont on note $A_j$ la colonne $j$ et 
un vecteur $y \in \bfA^n$, on a la formule dite de Cramer :
$$
\det(A) \, y 
\ = \ 
\sum_{j=1}^n 
\det(A_1, \dots, A_{j-1}, y, A_{j+1}, \dots, A_n) \, A_j
$$
qui traduit l'égalité matricielle bien connue $A \widetilde A = \det(A) \Id_n$
où $\widetilde A$ est la matrice cotransposée de $A$.

\begin{prop}[Cramer] \label{CramerSymetrie}
A $n-1$ vecteurs $v_1, \dots, v_{n-1}$ de $\bfA^n$, on associe la forme
linéaire déterminantale $\Delta_{\uv} : \bfA^n \rightarrow \bfA$
définie par $\Delta_{\uv} = \det (v_1, \dots, v_{n-1}, \sbullet)$ (la
place de l'argument $\sbullet$ dans la définition de $\Delta_{\uv}$
est sans importance).

\begin{enumerate}[\rm i)]
\item 
Pour $x, y \in \bfA^n$, on dispose de l'appartenance:
$$
\Delta_{\uv}(x)\, y  \, - \, \Delta_{\uv}(y)\, x 
\ \in \ 
\sum_{j=1}^{n-1} \bfA v_j
$$

\item 
Soit $M \in \mathbb M_{n,m}(\bfA)$ ayant ses mineurs d'ordre $n$ nuls.
Pour n'importe quelle famille $\uv = (v_1, \dots, v_{n-1})$ de $n-1$ vecteurs colonnes de $M$, 
la forme linéaire $\Delta_{\uv}$ est nulle sur $\Im M$, ou encore $\Delta_{\uv} \in \Ker\transpose M$.
\end{enumerate}
\end{prop}

\begin{proof} \leavevmode

i)
On utilise la formule de Cramer avec $A_j = v_j$ pour $j \leqslant n-1$ et $A_n = x$.
Le membre gauche vaut $\det(A) y = \det (v_1, \dots, v_{n-1}, x) y = \Delta_{\uv}(x) y$.
Le dernier terme de la somme vaut $\det(A_1, \dots, A_{n-1}, y) A_n = \Delta_{\uv}(y)x$.

\smallskip
ii)
Résulte de la définition de $\Delta_{\uv}$.
\end{proof}

\begin {rmq}
L'appartenance dans le point i) de la proposition précédente est un cas particulier
d'un résultat plus général qui fait l'objet du corollaire du théorème~\ref{ThProportionnalite}
(relations de Plücker-Sylvester).
Il peut être énoncé ainsi: soient $p,q$ deux entiers complémentaires à $n$, et des vecteurs
$u_1, \cdots, u_p,\ x_1, \cdots, x_q,\ y_1, \cdots, y_q \in \bfA^n$.
Alors le $q$-vecteur $\bfz$ défini comme la différence
$$
\bfz = \det(u_1,\cdots,u_p, x_1,\cdots,x_q)\,y_1\wedge\cdots\wedge y_q   -
\det(u_1,\cdots,u_p, y_1,\cdots,y_q)\,x_1\wedge\cdots\wedge x_q
$$
a la propriété suivante:
$$
\bfz \ \in\ \sum_{j=1}^p u_j\,\wedge\,\BW^{q-1}(\bfA^n)
$$
Le point i) de la proposition, à la notation $v \leftrightarrow u$ près, est le cas particulier
$p = n-1$, $q=1$. L'appartenance ci-dessus est en lien avec la notion d'isomorphisme
déterminantal et le théorème de proportionnalité.
\end {rmq}

\subsection{Dualité en algèbre extérieure et isomorphismes déterminantaux}
\label{ExteriorAlgebraDuality}

Ici et ailleurs, nous allons devoir régler des histoires de gauche/droite,
en particulier faire le choix d'isomorphismes déterminantaux
$\sharp : \BW^d(\bfA^n) \simeq \BW^{n-d}\big((\bfA^n)^\star\big)$ et de différentielles
des complexes de Koszul, ces deux notions
n'étant pas indépendantes.
Nous ne tenons pas à suivre la stratégie d'Eisenbud (Commutative Algebra), prop. 17.15
(auto-duality of the Koszul-complex) dont la preuve est à la charge
du lecteur avec la mention \og The identification of 
$\BW^d(\bfA^n)$ and $\BW^{n-d}\big(\bfA^n\big)^\star$ require some signs\fg.

\medskip

Nous souhaitons donner un descriptif le plus précis possible mais
certains résultats seront fournis sans preuve.  Précis car la notion
d'isomorphisme déterminantal va intervenir à de nombreuses reprises:
auto-dualité du complexe de Koszul, théorème de proportionnalité,
définition de la forme linéaire associée à un module librement
résoluble de rang 1, etc. Le théorème de proportionnalité met
directement en jeu la notion d'isomorphisme déterminantal, en voici un
énoncé simple.  Pour deux entiers $p,q$ complémentaires à~$n$, $p+q =
n$, identifions provisoirement $\bfA^n$ et son dual et désignons par
$\sharp : \BW^p(\bfA^n) \simeq \BW^{q}(\bfA^n)$ un isomorphisme
déterminantal. Par exemple:
$$
\sharp : \quad 
x_1\wedge\cdots\wedge x_p \longmapsto
\sum_{1\le j_1 < \cdots < j_q \le n} \kern -5pt
\det(x_1, \cdots, x_p,\ e_{j_1}, \cdots, e_{j_q})\, e_{j_1}\wedge\cdots\wedge e_{j_q}
$$
Etant donnés des vecteurs $u_1, \cdots, u_p, v_1,
\cdots, v_q$ de $\bfA^n$ vérifiant $\scp{u_i}{v_j} = 0$ pour
tous~$i,j$, le théorème de proportionnalité affirme que
les deux vecteurs de $\BW^q(\bfA^n)$
$$
(u_1 \wedge \cdots \wedge u_p)^\sharp  \quad \text{et} \quad
v_1 \wedge \cdots \wedge v_q
\qquad \text{sont proportionnels}
$$
Il s'agit d'un résultat très important car il participe à la mise en place de la
structure multiplicative des résolutions libres finies (confer le chapitre~\ref{ChapStructureMultiplicative}
et l'annexe~\ref{sectionAnnexeStructureMultiplicative}).
Nous fournirons en remarque~\ref{rmqProportionnalite} une preuve dans
le cas particulier $(p,q) = (n-1,1)$, preuve qui illustre le rôle joué par la formule
de Cramer de la section précédente.

\medskip
Opérons pour l'instant de manière informelle en considérant un $\bfA$-module quelconque $L$
et une forme linéaire $\mu \in L^\star$. A partir de ces données, nous pouvons
fabriquer une anti-dérivation $\partial = \partial^\mu : \BW^\sbullet(L) \to
\BW^{\sbullet-1}(L)$ de la manière suivante où les $x_i$ sont dans $L$:
$$
\begin {array}{lll}
\partial(x_1) &=& \mu(x_1) \\
\partial(x_1 \wedge x_2) &=& \mu(x_1)x_2 - \mu(x_2)x_1 \\
\partial(x_1\wedge x_2\wedge x_3) &=&
\mu(x_1)\,x_2\wedge x_3 - \mu(x_2)\,x_1\wedge x_3 + \mu(x_3)\,x_1\wedge x_2
\end {array}
$$
Et ainsi de suite. La règle est claire: on procède de la gauche vers la droite en
alternant les signes, de gauche vers la droite étant le choix que nous
avons retenu. On pourrait bien entendu procèder de manière analogue de
la droite vers la gauche. On laisse le soin au lecteur de vérifier que
$\partial$ est bien définie et vérifie, pour deux éléments homogènes
$\bfx, \bfy$ de~$\BW^\sbullet(L)$:
$$
\partial(\bfx \wedge\bfy) = \partial(\bfx)\wedge\bfy + J(\bfx)\wedge\partial(\bfy)
\qquad  \text{avec}\quad J(\bfx) = (-1)^{\deg \bfx}\,\bfx
$$
On dit alors que $\partial$ est une anti-dérivation à gauche de degré $-1$, disons que c'est la
terminologie retenue par C.~Chevalley dans 
Algèbre, réédition du chapitre~III, Observations de Chevalley, rédaction \no186, nbr~089
(archives de Bourbaki). Le choix de la droite vers la gauche aurait conduit à
une anti-dérivation à droite (C.~Chevalley encore):
$$
\partial(\bfx \wedge\bfy) = \partial(\bfx)\wedge J(\bfy) + \bfx\wedge\partial(\bfy)
\qquad  \text{avec}\quad J(\bfy) = (-1)^{\deg \bfy}\,\bfy
$$
Précisons que $\partial$ est l'\emph{unique} anti-dérivation
à gauche de degré $-1$ prolongeant $\mu : \BW^1(L) = L \to \BW^0(L) = \bfA$.
Nous devons lui trouver un nom en fonction de $\mu$ mais ce nom est tout trouvé
car il existe déjà; il s'agit du produit intérieur droit $\sbullet \intd \mu : 
\BW^\sbullet(L) \to \BW^{\sbullet-1}(L)$. On a donc \og par définition\fg:
$$
(x_1 \wedge \cdots \wedge x_d) \intd \mu \ = \
\sum_{i=1}^d \,(-1)^{i-1} \, \mu(x_i) \, 
x_1 \wedge \cdots \wedge \backslash \kern -6pt{x_i} \wedge \cdots \wedge x_d
$$
Dans la suite, il n'est nul besoin d'être un expert en produits intérieurs droit/gauche
ni même de connaître ces notions. Signalons que
$\BW^1(L^\star) = L^\star$ opère à droite sur $\BW(L)$ via $\intd$
$$
\BW^\sbullet (L) \times \BW^1(L^\star)  \to \BW^{\sbullet-1}(L), \qquad
(\bfx, \mu)  \to \bfx \intd \mu
$$
Règle en ce qui concerne $\intd$: l'élément qui opère sur l'autre
est celui placé à l'extrémité \emph {libre} du trait horizontal. Ainsi
dans $\text{truc}\intd\text{machin}$, c'est 
$\text{machin}$ (une forme linéaire sur $L$) qui opère sur $\text{truc}$
(un habitant de l'algèbre extérieure de $L$) pour donner un résultat de
même nature que $\text{truc}$ (un habitant de l'algèbre extérieure de~$L$).

\label{NOTA01-muProduitInterieur}
\index{anti-dérivation à gauche}%
\index{produit intérieur droit}%

\begin {lem} \label{AntiDerivationCarreNul}
Soit $g$ une anti-dérivation à gauche de degré $-1$ de $\BW^\sbullet(L)$. Alors $g \circ g = 0$.
\end {lem}

\begin {proof}
Soit $x \in L$ et $\bfy \in \BW^d(L)$.
$$
g(x\wedge\bfy) = g(x)\wedge\bfy - x\wedge g(\bfy) = g(x)\bfy - x\wedge g(\bfy) 
$$  
si bien que
$$
(g\circ g)(x\wedge\bfy) = g(x)g(\bfy) - \big(g(x)g(\bfy) - x\wedge (g\circ g)(\bfy)\big) =
x\wedge (g\circ g)(\bfy)
$$
Donc si $g\circ g$ est nulle sur $\BW^d(L)$, alors elle est nulle sur $\BW^{d+1}(L)$.
Or $g$ étant de degré $-1$, elle est nulle sur $\BW^0(L)$; à fortiori $g\circ g$
est nulle sur $\BW^0(L)$ donc nulle partout.
\end {proof}
  
\smallskip

La dualité mise en place ci-après va nous permettre de voir que  les
endomorphismes $\sbullet\intd\mu$ et $\mu\wedge\sbullet$ sont adjoints
l'un de l'autre. De plus, lorsque $L$ est libre de rang fini,
ils sont conjugués par un isomorphisme déterminantal bien choisi:
c'est ce que l'on appelle l'auto-dualité du complexe de Koszul.

\subsubsection{Dualité canonique $\BW^d(L) \times \BW^d(L^\star) \to \bfA$ pour un module quelconque $L$}

\label{NOTA01-DualiteCanonique}%

Pour tout $\bfA$-module $L$, la dualité canonique entre $L$ et $L^\star$ 
définie par $\langle x \, | \, \mu \rangle = \mu(x)$
se prolonge en une dualité naturelle entre les puissances extérieures via :
$$
\BW^d(L) \times \BW^d\big(L^\star \big) 
\ \longrightarrow \ 
\bfA
\qquad  \qquad 
(x_1 \wedge \cdots \wedge x_d, \ \mu_1 \wedge \cdots \wedge \mu_d)
\ \longmapsto \ 
\det \mu_i(x_j)
$$
Le fait que $L^\star$ soit le dual de $L$ est de peu d'importance: pour
deux modules quelconques $M, M'$, n'importe quelle dualité $M \times M' \to \bfA$,
notée $(m,m') \to \scp{m}{m'}$, sans aucune autre propriété que la bilinéarité,
se prolongerait de la même manière en une application bilinéaire
$\BW^d(M) \times \BW^d(M') \to \bfA$.

\smallskip
Comme toute application bilinéaire $E \times F \to \bfA$ donne naturellement naissance
à deux applications linéaires $E \to F^\star$ et $F \to E^\star$,
on dispose en particulier d'une application canonique:
$$
\textstyle
\begin{array}[t]{rcl}
\BW^d (L^\star) & \longrightarrow & \big(\BW^d L\big)^\star  \\ [0.5em]
\mu_1 \wedge \cdots \wedge \mu_d & \longmapsto & 
\langle\sbullet \mid \mu_1\wedge\cdots\wedge\mu_d\ \rangle
\end{array}
$$
Les deux applications linéaires $\sbullet\intd \mu : \BW^\sbullet(L) \to \BW^{\sbullet-1}(L)$ et
$\mu\wedge \sbullet : \BW^\sbullet(L^\star) \to \BW^{\sbullet+1}(L^\star)$ sont adjointes
l'une de l'autre pour la dualité ci-dessus:
$$
\framebox [1.1\width][c]{$\scp {\bfx\intd\mu}{\bsnu} = \scp {\bfx}{\mu\wedge \bsnu}$}
\qquad
\bfx \in \BW^{d+1}(L), \quad \bsnu \in \BW^{d}(L^\star)
$$
Il est important de remarquer que dans l'égalité encadrée, les symboles $\bfx, \mu, \bsnu$
ne changent pas de place: on remplace simplement $\intd\mu\mid$ par $\mid\,\mu\;\wedge$.
Insistons sur le fait que cette égalité encadrée ne nécessite aucune hypothèse sur $L$.
Une manière de montrer cette formule consiste à prendre
$\bfx = x_1 \wedge \cdots \wedge x_{d+1}$ et $\bsnu = \mu_1 \wedge \cdots \wedge \mu_d$ puis à développer
le déterminant de droite selon la première ligne. Par exemple, pour $d = 2$,
$\bfx = x_1 \wedge x_2 \wedge x_3$ et $\bsnu = \mu_1 \wedge \mu_2$:
$$
\scp {\bfx}{\mu \wedge \bsnu} =
\begin {vmatrix}
\mu(x_1) & \mu(x_2) & \mu(x_3) \\
\mu_1(x_1) & \mu_1(x_2) & \mu_1(x_3) \\
\mu_2(x_1) & \mu_2(x_2) & \mu_2(x_3) \\
\end {vmatrix}
$$
tandis que:
$$
\bfx \intd \mu =
\mu(x_1)\,x_2\wedge x_3 - \mu(x_2)\,x_1\wedge x_3 + \mu(x_3)\,x_1\wedge x_2
$$
conduisant au même résultat
$$
\scp{\bfx\intd \mu}{\bsnu} =
\mu(x_1) \begin {vmatrix} \mu_1(x_2) & \mu_1(x_3) \\ \mu_2(x_2) & \mu_2(x_3) \end {vmatrix}
-\mu(x_2) \begin {vmatrix} \mu_1(x_1) & \mu_1(x_3) \\ \mu_2(x_2) & \mu_2(x_3) \end {vmatrix}
+ \mu(x_3) \begin {vmatrix} \mu_1(x_1) & \mu_1(x_2) \\ \mu_2(x_1) & \mu_2(x_2) \end {vmatrix}
$$
Le caractère adjoint des deux applications linéaires $\mu\wedge\sbullet$ et $\sbullet\intd \mu$ 
se manifeste par la commutativité du diagramme dans lequel les verticales sont les morphismes canoniques:
$$
\xymatrix @R=1cm @C=2cm@M=0.4pc{
\BW^{d} \big(L^\star \big) \ar[r]^-{\textstyle \mu\wedge\sbullet}\ar[d]|-{\rm can.}&
    \BW^{d+1} \big(L^\star\big) \ar[d]|-{\rm can.}
\\
\big(\BW^{d} L\big)^\star \ar[r]^-{\textstyle\transpose{(\sbullet\intd\mu)}} &
  \big(\BW^{d+1} L\big)^\star
\\
}
$$

\subsubsection{L'isomorphisme canonique $\BW^d(L^\star) \overset{\mathrm {can.}}{\simeq} \big(\BW^d L\big)^\star$
lorsque $L$ est libre de rang fini}

Lorsque $L$ est libre de rang fini, l'application canonique
$\BW^d(L^\star) \to \big(\BW^d L\big)^\star$ est un isomorphisme. Pour le voir,
on fait le choix d'une base $(e_1, \cdots, e_n)$ de $L$, 
ce qui  fournit une base $(e_I)_I$ de $\BW^\sbullet(L)$ indexée par les parties $I$ de $\{1..n\}$:
$$
e_I = e_{i_1} \wedge \cdots \wedge e_{i_d}, \qquad   I = \{ i_1 < \cdots < i_d\}
$$
Notons $(e^\star_1, \cdots, e^\star_n)$ la base duale et $e^\star_I \in \BW(L^\star)$ défini par
$$
e^\star_I = e^\star_{i_1} \wedge \cdots \wedge e^\star_{i_d}
$$
Alors, au sens de la \og dualité déterminantale\fg:
$$
\scp{e_I}{e^\star_J} = \begin {cases} 0 &\text{si } I \ne J \\ 1&\text{ sinon} \\ \end{cases}
$$
si bien que $(e^\star_I)_I$ s'identifie à la base duale de $(e_I)_I$.

\label{NOTA01-eI}%

\medskip

Pour $\mu\in L^\star$, dans le diagramme commutatif, les verticales sont les isomorphismes canoniques:
$$
\xymatrix @R=1cm @C=2cm@M=0.4pc{
\BW^{d} \big(L^\star \big) \ar[r]^-{\textstyle \mu\wedge\sbullet}\ar[d]|-{\textstyle\overset{\rm can.}{\simeq}}&
    \BW^{d+1} \big(L^\star\big) \ar[d]|-{\textstyle\overset{\rm can.}{\simeq}}
\\
\big(\BW^{d} L\big)^\star \ar[r]^-{\textstyle\transpose{(\sbullet\intd\mu)}} &
  \big(\BW^{d+1} L\big)^\star
\\
}
$$

\subsubsection{L'isomorphisme déterminantal $\sharp_\bfe$ (Hodge droit) associé à une orientation $\bfe$ de $L$}

On désigne par $L$ un $\bfA$-module libre de rang $n$ orienté par
$\bfe\in \BW^n(L)$.  Cette orientation $\bfe$ sur $L$ induit une
orientation sur $L^\star$ de la manière suivante.  On note
$\bfe^\star\in \big(\BW^n L\big)^\star$ l'unique forme linéaire telle que
$\bfe^\star(\bfe) = 1$ que l'on identifie, via l'isomorphisme
canonique, à un élément de $\BW^n(L^\star)$. On note du même nom
$\bfe^\star$ l'orientation ainsi obtenue sur $L^\star$.  Si $(e_1,
\cdots, e_n)$ est une base de $L$ telle que $\bfe = e_1 \wedge \cdots
\wedge e_n$, alors en notant $(e_1^\star, \cdots, e_n^\star)$ la base
duale de $L^\star$, on a $\bfe^\star = e_1^\star \wedge \cdots \wedge
e_n^\star$.  Si $\bfe'$ est une autre orientation de $L$, on a $\bfe'
= \lambda\,\bfe$ avec $\lambda \in \bfA$ inversible et $\bfe'^\star =
\lambda^{-1}\,\bfe^\star$.

\label{NOTA01-OrientationDuale}%

\begin {defns} [Isomorphisme déterminantal]
\label{IsoDetOrientation}  
\leavevmode
\begin {enumerate}[\rm i)]
\item  
On désigne par $\sharp_e : \BW^\sbullet(L) \to \big(\BW^{n-\sbullet}L\big)^\star$ l'isomorphisme
défini par:
$$
\sharp_{\bfe} : 
\begin{array}[t]{rcl}
\BW^d(L) & \longrightarrow & \big(\BW^{n-d}L\big)^\star \\ [0.4em]
\bfx & \longmapsto & \oriented{\bfx\wedge\sbullet\,}_\bfe 
\end{array}
\qquad \qquad
0 \le d \le n
$$
\item
Un isomorphisme déterminantal entre $\BW^d(L)$ et $\big(\BW^{n-d}L\big)^\star \overset{\rm can.}{\simeq}
\BW^{n-d}(L^\star)$ est un isomorphisme de la forme $\lambda\,\sharp_\bfe$ pour un $\lambda$
inversible et $\bfe$ une orientation sur $L$. Par exemple, en changeant la position
du joker $\sbullet$, l'application $\bfx \mapsto \oriented{\sbullet\wedge\bfx\,}_\bfe$ est un isomorphisme
déterminantal.
\end{enumerate}  
\end{defns}

\label{NOTA01-sharp}%
\index{isomorphisme déterminantal}%

\medskip

Le lemme suivant apporte des précisions sur le fait que $\sharp_\bfe$ est un isomorphisme.
On rappelle la notion de \emph {signe étendu} défini de la manière suivante pour $i,j \in \bbZ$
et $I,J$ deux parties finies de $\bbZ$ (ou plus généralement d'un ensemble
totalement ordonné quelconque):
$$
\varepsilon(i,j) = 
\begin {cases} 0 &\text{si } i=j\\ 1 &\text{si } i<j\\ -1 &\text{si } i>j\\ \end {cases}
\qquad\qquad
\varepsilon(I,J) = \prod_{\substack{i\in I\\ j \in J}} \varepsilon(i,j)
$$
\label{NOTA01-epsilon}%
\index{signe étendu}%
On fait le choix d'une base $e = (e_1, \cdots, e_n)$ de $L$ telle que $\bfe=e_1 \wedge\cdots\wedge e_n$.
Ceci fournit une base $(e_I)_I$ de $\BW^\sbullet(L)$ indexée par les parties $I$ de $\{1..n\}$,
base qui vérifie:
$$
e_I\wedge e_J = \varepsilon(I,J)\, e_{I \cup J}
$$
On notera $(e^\star_I)_I$ la base duale.

\begin {lem}
Dans le contexte ci-dessus, en notant $\bfx^\sharp$ au lieu de $\sharp_\bfe(\bfx)$.

\begin {enumerate} [\rm i)]
\item
Pour $\bfx \in \BW^d(L)$:
$$
\bfx^\sharp = \sum_{\#J=n-d} \oriented{\bfx\wedge e_J}_\bfe\, e_J^\star
$$  
\item
L'application $\sharp_\bfe$ réalise $e_I \mapsto \varepsilon(I,\overline I)\,e_{\overline I}^\star$ où $\overline I$
est le complémentaire de $I$ dans $\{1..n\}$.  C'est donc bien un isomorphisme.

\item
En utilisant les identifications canoniques $\BW^{n-d}(L)^\star \overset{\rm can.}{\simeq} \BW^{n-d}(L^\star)$
et $\BW^{d}(L^\star)^\star \overset{\rm can.}{\simeq} \BW^d(L)$, l'isomorphisme inverse de $\sharp_\bfe$ est donné
par:
$$  
\bsnu \longmapsto \oriented{\sbullet\wedge \bsnu}_{\bfe^\star},\qquad
\bsnu \in \BW^{n-d}(L^\star)  
$$
Note: le joker $\sbullet$ a changé de place en passant de la position droite à la position gauche.
\end {enumerate}  
\end {lem}

\begin {proof} \leavevmode

i) Par définition de la base duale
$$
\bfx^\sharp = \sum_{\#J=n-d} \scp{\bfx^\sharp}{e_J}\, e_J^\star
$$  
Et par définition de $\sharp_\bfe$, on a $\scp{\bfx^\sharp}{e_J} = [\bfx\wedge e_J]_\bfe$.

\medskip
ii)
Soit $d = \#I$. Pour $\#J = n-d$, on a $\scp{e_I^\sharp}{e_J} = \oriented{e_I\wedge e_J}_\bfe =
\varepsilon(I,J)$ qui est nul sauf pour $J=\overline I$.

\medskip
iii) Puisque $\sharp_\bfe$ réalise $e_I \mapsto \varepsilon(I,\overline I)\,e_{\overline I}^\star$,
son inverse, en utilisant la correspondance $J \leftrightarrow \overline I$, réalise
$e_J^\star \mapsto \varepsilon(\overline J, J)\,e_{\overline J}$.
En écrivant:
$$
\oriented{\sbullet\wedge \bsnu}_{\bfe^\star} = \sum_{\#I=d}\oriented{e_I^\star\wedge\bsnu}_{\bfe^\star}\, e_I
$$
on voit que $\bsnu \mapsto \oriented{\sbullet\wedge \bsnu}_{\bfe^\star}$ réalise également
$e_J^\star \mapsto \varepsilon(\overline J, J)\,e_{\overline J}$.
\end {proof}

\begin {rmq}[Cas particulier du théorème de proportionnalité]
\label{rmqProportionnalite}
Soit $L = \bfA^n$ muni de sa base canonique $(e_i)_{1\leqslant i\leqslant n}$,
$\bfe = e_1 \wedge \cdots  \wedge e_n$
et $n-1$ vecteurs $u_1, \cdots, u_{n-1}$ de $\bfA^n$.
Alors la forme linéaire $\sharp_{\bfe}(u_1 \wedge \cdots \wedge u_{n-1})$ n'est autre
que la forme notée~$\Delta_{\uu}$ dans la proposition \ref{CramerSymetrie}.
On dit parfois que $\Delta_{\uu}$ est la forme linéaire des mineurs signés de
la matrice $U : \bfA^{n-1} \to \bfA^n$ dont les $n-1$ colonnes
sont les $u_i$.

\smallskip
En identifiant $\bfA^n$ et son dual, 
nous allons en déduire un cas particulier du théorème de proportionnalité.
Nous voyons donc la forme $\Delta_{\uu} = \sharp_{\bfe}(u_1 \wedge \cdots \wedge u_{n-1})$
comme le vecteur $w \in \bfA^n$ de composantes:
$$
w_i = \Delta_{\uu}(e_i) \overset{\rm def.}{=} \det(u_1,\cdots,u_{n-1}, e_i) \qquad 1 \le i \le n
$$
Le résultat à prouver est le suivant: si $v \in \bfA^n$ vérifie
$\scp{u_j}{v} = 0$ pour $1\le j\le n-1$, alors $w$ et $v$ sont
proportionnels i.e. $w_k v_\ell = w_\ell v_k$ pour tous $k,\ell$. Nous
allons montrer beaucoup mieux sous la forme d'une identité algébrique
sans aucune hypothèse sur $v$: chaque $w_k v_\ell - w_\ell v_k$ est
une combinaison linéaire (explicite) des produits scalaires $\scp{u_j}{v}$. 

\smallskip
Introduisons $w_ke_\ell - w_\ell e_k \overset{\rm def.}{=} \Delta_{\uu}(e_k)e_\ell -
\Delta_{\uu}(e_\ell)e_k$. D'après Cramer (cf \ref{CramerSymetrie}), il
y a des $\lambda_j \in \bfA$ tels que 
$w_ke_\ell - w_\ell e_k = \sum_{j=1}^{n-1} \lambda_j u_j$.
D'où, par produit scalaire avec $v$, l'obtention de la combinaison convoitée:
$$
\scp{w_ke_\ell - w_\ell e_k}{v} = w_kv_\ell - w_\ell v_k = \sum_{j=1}^{n-1} \lambda_j \scp{u_j}{v}
$$
A noter le rôle mineur joué par le vecteur $v$ et a contrario l'importance de la
propriété de Cramer $\Delta_{\uu}(x)y - \Delta_{\uu}(y)x \in \sum_j \bfA u_j$
quels que soient $x,y \in\bfA^n$.
\end {rmq}

\subsubsection{Conjugaison entre $\sbullet\intd\mu$ et $\mu\wedge\sbullet$
via un isomorphisme déterminantal bien choisi}
\label {AutoDuality}

Ici $L$ est un module libre de rang $n$ orienté par $\bfe$,
$\mu \in L^\star$ et $d,d'$ deux
entiers complémentaires à~$n$, $d+d' = n$. On veut conjuguer
$\sbullet\intd\mu$ et $\mu\wedge\sbullet$ via des isomorphismes déterminantaux
notés ? dans le diagramme de gauche. Dans la suite, ce qui est équivalent, on
travaille avec le \emph {diagramme de droite}:
$$
\xymatrix @C= 1.5cm@M=0.6pc{
\BW^d L \ar[r]^-{\sbullet\intd\mu} \ar[d]_? & \BW^{d-1} L \ar[d]^? \\
\BW^{d'}(L^\star) \ar[r]^-{\mu\wedge\sbullet} & 
\BW^{d'+1}(L^\star) \\
}
\qquad\qquad
\xymatrix @C= 1.5cm@M=0.6pc{
\BW^d L \ar[r]^-{\sbullet\intd\mu} \ar[d] & \BW^{d-1} L \ar[d] \\
\big(\BW^{d'} L \big)^\star \ar[r]^-{\kern-2pt \transpose{(\sbullet\intd \mu)}} & 
\big(\BW^{d'+1} L \big)^\star \\
}
$$
Pour les verticales, \emph {essayons} $\sharp_\bfe : \bfx \longmapsto
\oriented {\bfx\wedge \sbullet}_\bfe$ en notant $\partial$ pour
$\sbullet\intd \mu$.  Soit $\bfx \in \BW^d L$ en haut à gauche. On
descend avec l'application $\oriented{\bfx\wedge\sbullet}_\bfe$ et on
poursuit à l'horizontale pour tomber sur la forme linéaire
$\oriented{\bfx\wedge\partial(\sbullet)}_\bfe$.  Si maintenant on
parcourt le chemin dans l'autre sens (l'horizontale puis la
verticale), on tombe sur
$\oriented{\partial(\bfx)\wedge\sbullet}_\bfe$.

Pour que cela commute, on \emph{aimerait} donc avoir 
$\bfx\wedge\partial(\bfy)  = \partial(\bfx) \wedge \bfy$
pour tout $\bfy \in \BW^{d' +1}L$ mais ce n'est pas le cas. Il s'introduit
un signe comme on le voit en utilisant le fait que $d+(d'+1) = n+1$ donc
$\bfx \wedge \bfy = 0$ et le caractère anti-dérivation à gauche de degré $-1$
de $\partial$:
$$
0 = \partial(\bfx\wedge\bfy) = \partial(\bfx)\wedge\bfy + (-1)^{\deg \bfx}\,\bfx\wedge\partial(\bfy) =
\partial(\bfx) \wedge \bfy + (-1)^d\, \bfx \wedge\partial(\bfy)
$$
de sorte que
$$
\partial(\bfx) \wedge \bfy = (-1)^{d-1}\,\bfx\wedge \partial(\bfy)
$$
Introduisons alors un signe $\varepsilon_d$ dans le morphisme vertical 
$\BW^d L \ \longrightarrow \ \big( \BW^{d'} L \big)^\star,\ 
\bfx\ \longmapsto \ \varepsilon_d \,\oriented{\bfx\wedge\sbullet}_\bfe$ 
où 
$\varepsilon_d$ est à ajuster pour avoir
$\varepsilon_d \,\oriented{\bfx \wedge \partial(\bfy)}_\bfe =
\varepsilon_{d-1} \,\oriented{\partial(\bfx) \wedge\bfy}_\bfe$.
Ainsi il suffit d'imposer la relation $\varepsilon_d  = (-1)^{d-1}\varepsilon_{d-1}$. 
Le signe $\varepsilon_d = (-1)^{1+2 + \cdots + d-1} = (-1)^{\frac{d(d-1)}2}$ convient.

\subsubsection{Retour (informel) sur le produit intérieur droit}

\index{produit intérieur droit}%


Ici $L$ est un module quelconque. On a vu que l'on pouvait faire
opérer $\BW^1(L^\star)$ à droite sur $\BW(L)$ via~$\intd$ en
définissant $\sbullet \intd \mu$ lorsque $\mu$ est de degré 1
i.e. $\mu \in L^\star$.  Bien entendu, on aurait pu faire opérer de
manière analogue (à droite) $\BW^1(L)$ sur $\BW(L^\star)$. Bien que
cela ne soit pas directement utile pour la suite, nous tenons à
préciser que l'on peut prolonger ces opérations de produit intérieur
droit à toute l'algèbre extérieure car il nous semble que cela apporte
un certain éclairage intrinsèque vu que le module $L$ est quelconque.

\medskip
Nous allons juste évoquer comment il faut s'y prendre.
Pour $\mu \in \BW^1(L^\star)$, notons $g_\mu$ l'anti-dérivation à gauche  $\sbullet\intd\mu$.
On a vu que $g_\mu \circ g_\mu = 0$, cf. le lemme \ref{AntiDerivationCarreNul}.
On en déduit, pour $\mu, \mu' \in L^\star$, en développant $g_{\mu + \mu'} \circ g_{\mu + \mu'} = 0$
que $g_\mu$ et $g_{\mu'}$ anti-commutent:
$$
g_\mu \circ g_{\mu'} + g_{\mu'} \circ g_\mu = 0
$$
Ceci permet de faire $\intd$-agir $\BW^2(L^\star)$ à droite sur $\BW(L)$
$$
\BW^\sbullet(L) \times \BW^2(L^\star) \to \BW^{\sbullet - 2}(L)
\qquad \text{via} \qquad
\bfx \intd (\mu \wedge \mu') := (\bfx \intd \mu) \intd \mu' 
$$
Pour $\bsmu \in \BW^2(L^\star)$, le caractère adjoint entre $\sbullet\intd\bsmu$ et $\bsmu\wedge\sbullet$
demeure:
$$
\framebox [1.05\width][c]{$\scp{\bfx\intd\bsmu}{\bsnu} = \scp{\bfx}{\bsmu\wedge\bsnu}$},
\qquad\quad
\bfx \in \BW^d(L), \quad \bsnu \in \BW^{d-2}(L^\star)
$$
comme on le voit en prenant $\bsmu = \mu\wedge\mu'$ et en écrivant
$$
\scp {\bfx\intd (\mu\wedge\mu')}{\bsnu} \overset{\rm def.}{=}
\scp {(\bfx\intd\mu)\intd\mu')}{\bsnu} =
\scp {\bfx\intd\mu}{\mu'\wedge\bsnu} =
\scp{\bfx}{\mu\wedge\mu'\wedge\bsnu}
$$
De manière plus générale, on peut faire $\intd$-agir à droite $\BW(L^\star)$ sur $\BW(L)$
de sorte que $\bfx\intd\bsnu \in \BW^{q-p}(L)$ pour $\bsnu \in \BW^p(L^\star)$,
$\bfx \in \BW^q(L)$, avec les propriétés
$$
\bfx \intd (\bsnu \wedge \bsnu') = (\bfx \intd \bsnu) \intd \bsnu'
\qquad \qquad
\scp {\bfx \intd \bsnu}{\bsnu'} = \scp{\bfx}{\bsnu \wedge \bsnu'}
$$
Et de manière analogue, on peut faire $\intd$-agir à droite $\BW(L)$ sur $\BW(L^\star)$.

\label{NOTA01-ProduitInterieurDroit}%
%
%

\bigskip

Nous sommes maintenant en mesure de préciser, pour un module libre $L$
de rang $n$ orienté par $\bfe$, pourquoi l'isomorphisme déterminantal
$\sharp_\bfe$ que nous avons retenu s'identifie à l'isomorphisme Hodge
droit $\Hd_{\bfe^\star}^{\,\llcorner} = \bfe^\star \intd \sbullet$
défini par l'orientation $\bfe^\star \in \BW^n(L^\star)$
déduite de $\bfe$.

Cet isomorphisme de Hodge $\Hd_{\bfe^\star}^{\,\llcorner} : \BW^dL \to \BW^{n-d}(L^\star)$
est défini à partir du produit intérieur droit par $\bfx \mapsto \bfe^\star \intd \bfx$.
Le fait qu'il coïncide avec $\sharp_\bfe :  \BW^d L \to \BW^{n-d}(L)^\star$ 
résulte du calcul suivant pour $\bfx\in \BW^d L$, \ $\bfy \in \BW^{n-d}L$ :
$$
\scp {\bfe^\star\intd\bfx}{\bfy} \ = \ \scp {\bfe^\star}{\bfx \wedge \bfy} 
\ =\ 
\oriented{\bfx \wedge \bfy}_\bfe  
$$

\bigskip

Nous pouvons également revenir sur la conjugaison entre $\mu\wedge\sbullet$ et
$\sbullet\intd\mu$ vue à la page~\pageref{AutoDuality} et de nouveau
expliquer la correction de signe à apporter dans la mise au point
des isomorphismes verticaux. Ici, nous procédons de manière légèrement différente par rapport au diagramme
de la page~\pageref{AutoDuality} en changeant d'une part le sens des isomorphismes
déterminantaux verticaux et d'autre part en corrigeant l'isomorphisme
déterminantal vertical $\bsnu \mapsto \bfe\intd\bsnu$ 
(attention également à l'utilisation de $d+1$ versus~$d$).

Nous allons voir pourquoi il y a \emph {commutation au signe prés} dans le diagramme de gauche.
On part du coin haut-gauche avec $\bsnu \in \BW^d(L^\star)$.
En allant à droite puis en descendant, on tombe sur $\bfe\intd(\mu\wedge\bsnu)$
tandis qu'en descendant puis en allant à droite on tombe sur
$(\bfe\intd\bsnu)\intd\mu$. Or
$$
(\bfe\intd\bsnu)\intd\mu = \bfe \intd(\bsnu \wedge \mu) = (-1)^d\,\bfe\intd(\mu\wedge\bsnu)  
$$
On corrige donc (à droite) les isomorphismes verticaux par un signe $\varepsilon_d$ de manière
à avoir $\varepsilon_{d+1} = (-1)^d \varepsilon_d$:
$$
\xymatrix @R=1.5cm @C=2cm@M=0.4pc{
\BW^{d} \big(L^\star \big) \ar[r]^-{\textstyle \mu\wedge\sbullet}\ar[d]^\simeq_{\bfe\intd\sbullet}&
    \BW^{d+1} \big(L^\star\big) \ar[d]_\simeq^{\bfe\intd\sbullet}
\\
\BW^{d'}(L) \ar[r]^{\textstyle\sbullet\intd\mu} & \BW^{d'-1}(L)
\\
}
\qquad\qquad
\xymatrix @R=1.5cm @C=2cm@M=0.4pc{
\BW^{d} \big(L^\star \big) \ar[r]^-{\textstyle \mu\wedge\sbullet}\ar[d]^\simeq_{\varepsilon_d(\bfe\intd\sbullet)}&
    \BW^{d+1} \big(L^\star\big) \ar[d]_\simeq^{\varepsilon_{d+1} (\bfe\intd\sbullet)}
\\
\BW^{d'}(L) \ar[r]^{\textstyle\sbullet\intd\mu} & \BW^{d'-1}(L)
\\
}
$$
Avec $\varepsilon_0=1$, on tombe sur:
$$
\varepsilon_d = (-1)^{1 + \cdots + d-1} = (-1)^{\frac{d(d-1)}{2}}
$$

\subsection{Suites régulières et suites $1$-sécantes}

Un certain nombre de notions de profondeur sont indispensables à l'étude du résultant, du moins
à l'approche retenue ici. Citons par exemple le théorème de Wiebe, la notion
de profondeur~2, le contrôle de la profondeur par les composantes homogènes dominantes,
le théorème de Hilbert-Burch etc.

Les auteurs ont été confrontés au dilemme suivant: impossible de tout admettre mais
également impossible de tout refaire. Nous avons opté pour le compromis
qui vient en essayant de présenter un exposé se voulant cohérent: donner les définitions
indispensables ainsi que les énoncés des résultats utilisés, en fournissant parfois
les preuves mais pas toujours. Procédant ainsi, cela permet à notre avis de faire
\og tourner les notions\fg{}. Nous avons choisi de faire une présentation sommaire
du complexe de Koszul tout en privilégiant parfois l'approche
élémentaire: par exemple, nous avons montré à la main qu'une suite $\ua$
régulière est 1-sécante alors qu'une autre voie plus rapide consiste
à utiliser l'acyclicité du complexe de Koszul descendant de $\ua$, fournissant
en particulier $\rmH_1(\ua) = 0$ i.e. le caractère 1-sécant de $\ua$.

\medskip

En cas d'absence de preuves, nous fournirons comme références
la thèse de C.~Tête \cite {Tete}. Le lecteur pourra compléter son information
en en consultant les premiers chapitres.
Il y a également le chapitre 1 de l'ouvrage de Bruns \& Herzog \cite {BrunsHerzog}
qui contient un traitement du complexe de Koszul.  Il
est cependant impératif de ne pas confondre la notion de profondeur
utilisée ici, qui est la même que celle présentée par  Bourbaki en
\cite{BourbakiACX} ou celle figurant dans l'ouvrage de Northcott~\cite{NorthcottFFR},
avec la notion de \og depth\fg{} spécifique au contexte local n\oe{}thérien.

\subsubsection*{Des éléments réguliers \og à vie\fg{} : les polynômes primitifs par valeur}

Il y a une classe importante d'éléments réguliers qui va intervenir dans la suite:
il s'agit de celle constituée des polynômes à plusieurs variables, à coefficients dans un
anneau commutatif quelconque~$\bfk$, et qui sont primitifs par valeur. Ce sont des polynômes
primitifs très particuliers.

On rappelle qu'un polynôme $f$ est \emph {primitif} si $1 \in \rmc(f)$
où  $\rmc(f)$, idéal contenu de $f$, est l'idéal de $\bfk$ engendré par les coefficients
de~$f$.  Un polynôme $f$ est dit \emph {primitif par valeurs} (au
pluriel) si $1$ est dans l'idéal engendré par les valeurs de $f$
prises sur $\bfk$ \idest{} lorsque l'on spécialise les indéterminées en toutes les valeurs
possibles des éléments de $\bfk$. Enfin, $f$ est dit \emph {primitif par
valeur} (au singulier) s'il y a une valeur de $f$ sur
$\bfk$ qui est inversible. Un polynôme primitif par valeurs est
primitif et évidemment un polynôme primitif par valeur l'est par
valeurs.  Il n'y a pas de réciproque. Soient les polynômes de~$\bbZ[a]$:
$$
f = a^2 + a = a(a+1), \qquad\qquad g = 5a - 2
$$
Alors $f$ est primitif mais pas par valeurs puisque $f(a)$ est pair pour tout $a \in \bbZ$.
Le polynôme $g$ est primitif par valeurs puisque $g(0) = -2$ et $g(1) = 3$ sont premiers
entre eux; mais il n'est pas primitif par valeur car l'équation $g(a) = \pm 1$ n'a pas de
solution $a$ dans $\bbZ$.

\index{polynôme!primitif}%
\index{polynôme!primitif par valeurs, par valeur}%

Ces propriétés se conservent par morphisme $\bfk \to \bfk'$.
Si $f$, à coefficients dans $\bfk$, est primitif, il en est de même du polynôme
à coefficients dans $\bfk'$, image de $f$ par le morphisme. Quand on est primitif,
c'est pour la vie. Idem en ce qui concerne les autres propriétés.

\medskip

Dans la suite, \og primitif par valeur\fg{} interviendra fréquemment dans le
contexte de la spécialisation d'un système en le jeu étalon (pour
comprendre les termes utilisés, cf. le chapitre~\ref{ObjetsSuiteP}).
Ici, on peut en donner une idée via l'exemple suivant dans lequel $a,b,c,d$ sont des
indéterminées sur un anneau commutatif. Considérons le polynôme
déterminant $2 \times 2$:
$$
f = \det(M),  \qquad   M = \begin {bmatrix}  a & b\\ c& d\\ \end {bmatrix}
$$
Puisque $\det(\Id_2) = 1$, la spécialisation $M \to \Id_2$ induit le fait que
$f$ est primitif par valeur.

\medskip

Le fait qu'un polynôme primitif soit régulier est un cas particulier d'un lemme dit
de McCoy dont on trouvera une preuve en \cite[Th. 7, section 5.3]{NorthcottFFR} ou bien 
en \cite[chap. III, corollaire 2.3]{LombardiQuitte}.

\begin{lem}[de McCoy]\label{McCoyPolyLemma}

Soit $f \in \bfk[\uT]$ un polynôme à plusieurs variables.  
Si son contenu $\rmc(f)$ est un idéal fidèle i.e. $\Ann_\bfk(\rmc(f)) = 0$,
alors le polynôme $f$ est un élément régulier de $\bfk[\uT]$.

A fortiori, un polynôme primitif est régulier et à fortiori,
un polynôme primitif par valeur est régulier.
\end{lem}

\label{NOTA01-contenu}%
\label{NOTA01-Ann}%
%
%

\subsubsection*{Suites régulières}

\begin{defn} 
Soit $\ua = (a_1, \cdots, a_n)$ une suite de scalaires et $E$ un module.
On dit que $\ua$ est $E$-régulière si pour tout $1 \le i \le n$, le scalaire $a_i$ est régulier modulo
$a_1E + \cdots + a_{i-1}E$.
Autrement dit, en notant $E_i = \sum_{1\le j < i}a_jE$, 
la multiplication $\xymatrix{E/E_i \ar[r]^-{\times a_i} & E/E_i}$  est injective.
En particulier, dans une suite $\ua$ $E$-régulière, le premier terme $a_1$ est $E$-régulier.

\smallskip
\noindent
Lorsque $E$ est l'anneau de base, on parle de suite régulière.

\smallskip
\noindent
La suite $\ua$ est dite commutativement régulière si la suite $(a_{\sigma(1)}, \cdots, a_{\sigma(n)})$
est régulière pour tout $\sigma \in \fS_n$.
\end{defn}

\index{suite!régulière}%
\index{suite!commutativement régulière}%

L'exemple le plus simple est probablement celui d'un anneau de polynômes en $n$
indéterminées $(X_1, \cdots, X_n)$ sur un anneau commutatif quelconque:
la suite des indéterminées $(X_1, \cdots, X_n)$ est une suite commutativement régulière.

\medskip
Le lemme suivant immédiat ne nécessite pas de preuve.

\begin{lem} \label{RegTrivialFacts} 
\leavevmode

\begin{enumerate}[\rm i)]
\item 
Un préfixe d'une suite régulière est encore une suite régulière
\item 
Une suite régulière sur $\bfA$ est régulière sur tout $\bfA$-module libre.
\end{enumerate}
\end{lem}

\subsubsection*{Permutabilité et suite régulière}

Une suite régulière ne reste pas forcément régulière après permutation 
de ses termes. 
Il y a le célèbre exemple de $\bfk[X,Y,Z]$ :
la suite $\bigl(X(Y-1),\, Y,\, Z(Y-1)\bigr)$ est régulière 
(et engendre l'idéal $\langle X,Y,Z \rangle$)
mais elle ne l'est plus si on permute les deux derniers termes. 
C'est cependant le cas si la suite est dans l'idéal maximal d'un anneau local noethérien
ou encore lorsque l'anneau est gradué 
et la suite constituée d'éléments homogènes de degré $> 0$.
Trouver dans la littérature cet énoncé en terrain gradué tel quel n'est pas si simple.
Il est souvent mis en parallèle avec son analogue en terrain local. 
Certains auteurs énoncent alors ce résultat de permutabilité en local sans
préciser d'hypothèse supplémentaire noethérienne, et ceci malgré le
contre-exemple célèbre de Dieudonné~\cite{Dieudonne} qui date de~1966 ! 
D'autres au contraire spécifient en terrain gradué l'hypothèse 
noethérienne alors qu'elle est totalement inutile.  
Enfin, dans certains ouvrages, la permutabilité (dans le contexte gradué) se déduit du
résultat suivant~: 
\begin{quote}
une suite constituée d'éléments homogènes de degré~$> 0$ 
est régulière si et seulement si elle est 1-sécante ou encore si
et seulement si elle est complètement sécante (ces deux dernières
notions, définies plus loin, sont invariantes par permutation).
\end{quote}
Le résultat ci-dessus est certes important mais la preuve est plus complexe que la 
permutabilité dont nous aurons besoin par exemple en~\ref{SuiteGeneriqueReguliere}
(pour prouver qu'une suite de polynômes homogènes génériques est régulière).
Voici un premier résultat de permutabilité dans un contexte assez courant.


\begin{prop} [Permutabilité d'une suite régulière en terrain homogène]
\label{PermutabiliteSuiteReguliereHomogene}
Soit $\bfA$ un anneau $\bbN$-gradué, $E$ un $\bfA$-module $\bbN$-gradué, et 
$(a_1,\ldots, a_n)$ une suite $E$-régulière 
{\rm constituée d'éléments homogènes de degré~$> 0$}. 
Alors $(a_1, \ldots, a_n)$ est une suite commutativement $E$-régulière.
\end{prop}

La preuve est relativement simple une fois le lemme suivant mis en place.

\begin{lem} \label{Permutabilite2elements}
Soit $\ua = (a_1, \ldots, a_n)$ une suite $E$-régulière.
La suite $\ua'$ obtenue en permutant $a_i, a_{i+1}$ est $E$-régulière
si et seulement si $a_{i+1}$ est $E/(a_1E + \cdots + a_{i-1}E)$-régulier.
\end{lem}

\begin{proof}
Dans la définition d'une suite régulière, 
les seuls changements vis-à-vis de $\ua$ qui interviennent pour 
$\ua' = (a_1, a_2, \dots, a_{i-1}, a_{i+1}, a_i, a_{i+2}, \dots, a_n)$
sont :

\begin{enumerate}[\rm (1)]
\item $a_{i+1}$ est $E/(a_1E + \cdots + a_{i-1}E)$-régulier ;
\item $a_i$ est $E/(a_1E + \cdots + a_{i-1}E + a_{i+1}E)$-régulier.
\end{enumerate}
La suite $\ua'$ est donc régulière si et seulement si les points (1) et (2) sont
vérifiés. Montrons qu'une suite régulière~$\ua$ vérifie toujours (2),
il en résultera que $\ua'$ est régulière si et seulement si le point (1)
est vérifié.

\smallskip

Pour montrer que $\ua$ vérifie le point (2), introduisons $E' = a_1E + \cdots + a_{i-1}E$. 
Soit $x \in E$ tel que $a_i x \in E' + a_{i+1}E$ ; 
nous devons montrer que $x \in E' + a_{i+1}E$.
Ecrivons l'hypothèse $a_i x \in E' + a_{i+1}E$ sous la forme $a_i x + a_{i+1}y \in E'$, 
que l'on relâche un instant en $a_{i+1} y \in E' + a_i E$.
Comme $(a_1, \dots, a_i, a_{i+1})$ est $E$-régulière, 
on a $y \in E' + a_i E$ que l'on écrit $y = y' + a_i z$ avec $y' \in E'$ et $z \in E$.
En reportant cela dans $a_i x + a_{i+1} y \in E'$, on obtient 
$a_i(x + a_{i+1}z) \in E'$.
Comme $(a_1, \ldots, a_i)$ est $E$-régulière, cela entraîne $x + a_{i+1} z \in E'$
donc $x \in E' + a_{i+1}E$.
\end{proof}

\begin{proof}[Preuve de~\ref{PermutabiliteSuiteReguliereHomogene}]
\leavevmode

$\rhd$
Il suffit de montrer qu'en permutant $a_i$ et $a_{i+1}$, la suite reste
$E$-régulière donc, d'après le lemme~\ref{Permutabilite2elements}  ci-dessus, que 
$a_{i+1}$ est $E_i$-régulier où l'on a posé $E_i = E/(a_1E + \cdots + a_{i-1}E)$.
On va travailler au niveau de~$E_i$. Le fait que $(a_1, \cdots, a_{i-1}, a_i, a_{i+1})$
soit $E$-régulière a comme conséquence que $(a_i, a_{i+1})$ est $E_i$-régulière
et, à partir de cette information, 
il s'agit de montrer que $a_{i+1}$ est $E_i$-régulier, 
ce qui nous ramène au point suivant en renommant $(E_i, a_i, a_{i+1})$ en $(F, a_1, a_2)$.

$\rhd$
Montrons maintenant sous l'hypothèse $(a_1, a_2)$ est une suite $F$-régulière homogène,
que $a_2$ est un élément $F$-régulier. Soit $x \in F$ tel que $a_2 x = 0$ 
et montrons que $x = 0$. Comme $a_2$ est homogène, on peut supposer $x$ homogène,
disons de degré $d$. 
A fortiori, on a $a_2 x = 0$ dans $F/a_1F$; 
comme $(a_1, a_2)$ est $F$-régulière, 
on a $x = 0$ dans $F/a_1F$ c'est-à-dire $x \in a_1F$ disons $x = a_1y$
où $y$ est homogène de degré $< d$ (car $a_1$ est homogène de degré $> 0$). 
On écrit $a_1a_2y = 0$, on simplifie par $a_1$ (qui est $F$-régulier) pour obtenir $a_2y = 0$. 
Une récurrence sur le degré des éléments du module permet de conclure.
\end{proof}

\subsubsection{Suites 1-sécantes}

\begin{defn} [Suite 1-sécante] \label {DefSuite1Secante}
On désigne par $(\varepsilon_i)_{1\le i\le n}$ la base canonique de $\bfA^n$.
On dit qu'une suite de scalaires $\ua = (a_1,\dots,a_n)$ est
1-sécante si pour toute relation $\sum_i u_i a_i = 0$ où les $u_i$ sont dans
$\bfA$, le relateur $\sum_i u_i\varepsilon_i \in \bfA^n$ est combinaison
linéaire des relateurs dits triviaux $a_i\varepsilon_j - a_j\varepsilon_i$.
\end {defn}

\index{suite!1-sécante}%

Cette notion est indépendante de l'ordre des termes contrairement à celle de suite régulière.
Une manière équivalente de raconter qu'une suite $\ua$ est 1-sécante consiste à introduire
les deux premières différentielles du complexe de Koszul descendant $\rmK_\sbullet(\ua)$ de $\ua$:
$$
\xymatrix {
\BW^2(\bfA^n) \ar[r]^-{\partial_2} & \BW^1(\bfA^n) = \bfA^n
}
\qquad
\partial_2(\varepsilon_i \wedge \varepsilon_j) = a_i\varepsilon_j - a_j\varepsilon_i
\qquad\qquad
\xymatrix @C = 1.5cm {
\partial_1 : \bfA^n \ar[r]^-{[a_1, \cdots, a_n]} & \bfA
}  
$$
de sorte que l'on a toujours $\partial_1 \circ \partial_2 = 0$. La suite est 1-sécante
si et seulement si $\Ker\partial_1 = \Im \partial_2$ \idest{} $\rmH_1(\ua) = 0$
lorsque l'homologie Koszul de $\ua$ sera en place. On peut voir $\BW^2(\bfA^n)$ comme
le module des matrices alternées $n \times n$ sur $\bfA$ via $\varepsilon_i \wedge \varepsilon_j
\leftrightarrow M_{ij}$ où $M_{ij} \in \bbM_n(\bfA)$ a tous ses coefficients nuls sauf celui
en position $(i,j)$ qui vaut $1$ et celui en position $(j,i)$ qui vaut $-1$. Avec cette
identification, $\partial_2$ n'est autre que $M \mapsto \ua\,M$ comme on le voit
en vérifiant que $\ua\,M_{ij} = a_i\varepsilon_j - a_j\varepsilon_i$.
Par exemple, pour $(i,j) = (3,1)$, on a
$$
\ua\, M_{3,1} = [a_1, a_2, a_3]
\begin {bmatrix}
0 & 0 & -1\\  
0 & 0 & 0\\
1 & 0 & 0\\
\end {bmatrix}
= [a_3, 0, -a_1] =
a_3 \varepsilon_1 - a_1\varepsilon_3 = \partial_2(\varepsilon_3 \wedge \varepsilon_1)
$$
Dire que $\ua$ est 1-sécante signifie donc que pour toute relation
$u_1 a_1 + \cdots + u_na_n=0$, il existe une matrice alternée $M$
de taille $n$ telle que $[u_1, \cdots, u_n] = [a_1, \cdots, a_n] M$.
Une suite 1-sécante de longueur 1 n'est autre qu'un élément régulier.
Chaque terme $a_i$ d'une suite 1-sécante est régulier modulo $\sum_{j \neq i} a_j\bfA$,
ceci s'expliquant par le fait que la diagonale d'une matrice alternée est nulle.

\medskip

On définit de manière analogue la notion de suite 1-sécante sur un $\bfA$-module $E$.

\begin{prop}[Le caractère 1-sécant d'une suite régulière] \label{regImplique1secante}
Soient $\ua = (a_1, \ldots, a_n)$ une suite de scalaires et $E$ un $\bfA$-module
tels que $\ua$ est $E$-régulière. Alors $\ua$ est 1-sécante sur $E$.

En particulier, chaque $a_i$ est régulier modulo $\sum_{j \neq i} a_j E$.
\end{prop}

\begin{proof}
Soit  une relation $a_1 x_1 + \cdots + a_n x_n = 0$ où les
$x_i \in E$;  nous devons montrer qu'il existe une matrice alternée
$M$ à coefficients dans~$E$ telle que $\ux = \ua\, M$.
On va le prouver par récurrence sur $n$.

$\rhd$ $n=2$ : $a_1 x_1 + a_2 x_2 = 0$. L'appartenance $a_2x_2 \in a_1E$ et la
$(E/a_1E)$-régularité de $a_2$ entraînent $x_2 \in a_1E$, disons $x_2 =
-a_1y$. Il vient $a_1 x_1 - a_2a_1y = a_1(x_1 - a_2y) = 0$ et puisque $a_1$
est $E$-régulier, on a $x_1 = a_2 y$, ce qui s'écrit~:
$$
\begin{bmatrix} x_1 & x_2 \end{bmatrix} 
\ = \
\begin{bmatrix} a_1 &  a_2 \end{bmatrix} 
\begin {bmatrix} 0 & -y\cr y& 0\cr \end {bmatrix}
$$

$\rhd$ On part de la relation $a_1x_1 + \cdots + a_n x_n = 0$.
En utilisant la $E / (a_1E + \cdots + a_{n-1}E)$-r\'{e}gularit\'{e} de~$a_n$, 
on en déduit que 
$x_n \in a_1E + \cdots + a_{n-1}E$, 
ce que l'on \'{e}crit $a_1 y_1 + \cdots + a_{n-1}y_{n-1} + x_n = 0$. 
En combinant les relations pour \'{e}liminer $x_n$, on
obtient~:
$$
a_1 (x_1 - a_n y_1) \ +\  \cdots \ +\ a_{n-1} (x_{n-1} - a_n y_{n-1}) 
\ = \ 0
$$
On utilise maintenant le fait que $(a_1, \ldots, a_{n-1})$ est $E$-régulière
et l'hypothèse de récurrence fournit une matrice $M$ alternée $(n-1) \times (n-1)$, \`{a}
coefficients dans~$E$, telle que~:
$$
\begin{bmatrix}
x_1 - a_ny_1 &\cdots & x_{n-1} - a_n y_{n-1}
\end{bmatrix}
 \ = \ 
\begin{bmatrix} a_1 & \cdots & a_{n-1} \end{bmatrix} 
M
$$
ou encore~:
$$
\begin{bmatrix} x_1 & \cdots & x_{n-1} \end{bmatrix} 
\ = \
\begin{bmatrix} a_1 & \cdots & a_{n-1} \end{bmatrix}  M 
\ +\ 
a_n 
\begin{bmatrix} y_1 & \cdots & y_{n-1} \end{bmatrix} 
$$
On peut alors \'{e}crire~:
$$
\begin{bmatrix} x_1 & \cdots & x_{n-1} & x_n \end{bmatrix} 
\ = \
\begin{bmatrix} a_1 & \cdots & a_{n-1} & a_n \end{bmatrix} 
\begin{bmatrix}
           &        &                & -y_1 \\
          &  M     &                & \vdots   \\ 
           &        &                & -y_{n-1} \\
 y_1       & \cdots & y_{n-1}        & 0
\end{bmatrix}
$$
et on a bien obtenu une matrice alternée $n \times n$ comme voulu.
\end{proof}

\subsubsection*{Suites régulières (reprise)}

\begin{prop}[Le $(a,b,ab)$-trick pour les suites régulières, utilisé en~\ref{MultiplicativiteResultant}] 
\label{ababRegTrick} 
Soient $E$ un $\bfA$-module, $a,b \in \bfA$ et $\uc = (c_1, \ldots, c_n)$ une suite de scalaires. 
Si les suites $(a,\uc)$ et $(b,\uc)$ sont $E$-régulières, il en est de même de la suite $(ab,\uc)$.
\end{prop}

\begin{proof}
Le produit de deux scalaires $E$-réguliers est $E$-régulier. 
Il reste à montrer qu'une égalité $c_i x = ab x_1 + c_2 x_2 + \cdots + c_{i-1} x_{i-1}$ 
avec $x, x_1, \ldots, x_{i-1} \in E$ entraîne $x \in ab E + \cdots + c_{i-1} E$. 
La suite $(a, c_2, \ldots, c_i)$ étant $E$-régulière, on dispose
d'une égalité :
$$
x = a y_1 + c_2 y_2 + \cdots + c_{i-1} y_{i-1} \qquad \hbox {avec des $y_j \in E$}
\leqno (\star)
$$
En multipliant par $c_i$ cette dernière égalité, et en comparant avec 
l'égalité initiale donnant $c_i x$, on en déduit :
$$
c_i(a y_1 + c_2 y_2 + \cdots + c_{i-1} y_{i-1}) 
\ =\  
ab x_1 + c_2 x_2 + \cdots + c_{i-1} x_{i-1}
$$
D'o\`u:
$$
a (c_i y_1 - bx_1) + c_2(c_iy_1 - x_2)  + \cdots + c_{i-1}(c_iy_{i-1} - x_{i-1}) = 0
$$
Puisque la suite $(a, c_2, \ldots, c_i)$ est $E$-régulière, elle est $1$-sécante ;
en particulier, le scalaire $a$
est régulier modulo $c_2E + \cdots + c_{i-1}E$ ; d'où
$$
c_i y_1 - bx_1 \, \in\, c_2E + \cdots + c_{i-1}E 
\qquad \hbox {et donc} \qquad
c_i y_1 \, \in\, bE + c_2E + \cdots + c_{i-1}E
$$
La suite $(b, c_2, \ldots, c_i)$ étant $E$-régulière, on obtient 
$y_1 \in bE + \cdots + c_{i-1}E$. 
En reportant cela dans $(\star)$, on obtient 
$$
x \,\in\, ab E + c_2 E + \cdots + c_{i-1} E
$$
ce qu'il fallait prouver.
\end{proof}

\medskip

Le résultat suivant sera utilisé dans le contrôle de la profondeur par les composantes
homogènes dominantes (cf. le théorème~\ref{ControleProfondeur}). Précisons qu'ici suite
extraite de $\ua$ désigne une suite de la forme $(a_{\sigma(1)}, \cdots, a_{\sigma(m)})$
où $m \le n$ et $\sigma \in \fS_n$. L'implication i) $\Rightarrow$ ii)  résulte de~\ref{regImplique1secante}.
En ce qui concerne la réciproque, nous renvoyons à la proposition V.7.2 de la
thèse de C.~Tête~\cite[Section V-7]{Tete}.

\begin{prop}[Suite commutativement régulière]
\label{ReguliereSuiteExtraite1Secante} 

Pour une suite $\ua = (a_1, \dots, a_n)$, les deux propriétés sont équivalentes:
\begin{enumerate}[\rm i)]
\item La suite $\ua$ est commutativement régulière.
\item Toute suite extraite de $\ua$ est $1$-sécante.
\end{enumerate}
\end{prop}

\subsection{Profondeur et complexe de Koszul}
\label{ProfondeurViaKoszul}


Etant donnée une suite de scalaires $\ua = (a_1, \cdots, a_n)$, on
veut définir la profondeur $\Gr(\ua)$,
qui est un entier de $\overline \bbN = \bbN \cup \{\infty\}$, ainsi que
sa version $\Gr(\ua\,;M)$ pour un module~$M$, et en fournir quelques
propriétés élémentaires.
La terminologie officielle pour $\Gr(\fa\,; M)$ est \og
profondeur du module $M$ relativement à l'idéal de type fini~$\fa$\fg{}.
Nous ne nous y conformerons pas et nous utiliserons plutôt, même
si c'est impropre, \og profondeur de $\ua$ sur $M$\fg{}.
Nous nous en excusons auprès des lecteurs.\footnote{Après vérification,
nous constatons que nous sommes en adéquation avec la terminologie de Northcott.
Celui-ci, en \cite[section 5.5, p. 149]{NorthcottFFR}, pour un idéal $\fb$
et un module $E$, utilise \emph {true grade or polynomial grade of $\fb$ on $E$.}}

\smallskip
En fait, nous allons nous concentrer essentiellement sur la signification d'inégalités
$\Gr(\ua) \ge k$ et $\Gr(\ua\,;M) \ge k$. Par exemple, si $\ua$ est unimodulaire
\idest{} $1 \in \langle\ua\rangle$, on a $\Gr(\ua) \ge k$ pour tout $k$ et
réciproquement. L'inégalité $\Gr(\ua) \ge \Gr(\ub)$ aura la
signification $\Gr(\ub) \ge k \Rightarrow \Gr(\ua) \ge k$ et l'égalité
$\Gr(\ua) = \Gr(\ub)$ la
signification $\Gr(\ub) \ge k \Leftrightarrow \Gr(\ua) \ge k$.

\medskip
Parmi les propriétés élémentaires, il est indispensable de disposer du fait que $\Gr(\ua)$
ne dépend que de l'idéal engendré par $\ua$, ce qui permet de définir $\Gr(\fa)$
pour un idéal $\fa$ de type fini. Nous aurons également besoin du comportement suivant:
$\fb \supseteq \fa \Rightarrow \Gr(\fb) \ge \Gr(\fa)$. Et de l'invariance par
exponentation: $\Gr(\fa) = \Gr(\fa^e)$ pour $e \ge 1$ dont une variante réside
dans l'implication $\Gr(\ua) \ge k \Rightarrow \Gr(a_1^{e_1}, \cdots, a_n^{e_n}) \ge k$.
Enfin, pour une idée intuitive de la profondeur, on verra que l'on a
$\Gr(\fa) \ge d$ lorsque $\fa$ contient une suite régulière de longueur~$d$.
La réciproque est fausse mais presque vraie (vérifiée après une extension polynomiale).
Ainsi, $\Gr(\fa) \ge 1$ ne signifie pas que $\fa$ contient un élément régulier mais
que $\Ann(\fa) = 0$. 
C'est équivalent à dire qu'il y a un système de générateurs $(a_0, \cdots, a_m)$ 
de $\fa$ tel que le polynôme $a_0 + a_1t + \cdots + a_mt^m \in \bfA[t]$ soit
régulier et c'est alors vérifié pour tout système générateur fini de $\fa$.

\medskip
Nous avons choisi de traiter cette notion de profondeur via le complexe de Koszul d'une
suite $\ua$. Nous suivrons de près le premier chapitre de la thèse de C.~Tête \cite{Tete}.
Mais le complexe de Koszul n'est pas le seul moyen d'étudier la profondeur, voir par exemple
les auteurs suivants dans lequel ce complexe n'apparaît pas: D.G. Northcott,
sections 5 et 6 de son ouvrage \cite{NorthcottFFR},
ou le mémoire ``Résolutions libres finies. Méthodes constructives'' \cite{CoqLomQuiTete}.
Les sections 2 à 6 de T. Coquand \& C. Quitté \cite {CoquandQuitte}
contiennent également une étude de la profondeur dans laquelle figure un
traitement mixte (complexe de Koszul \& suites régulières latentes).

\subsubsection{Les complexes de Koszul descendant et montant d'une
               forme linéaire $\mu : L \rightarrow \bfA$}

Soit $\ua = (a_1, \cdots, a_n)$ une suite de scalaires. L'objectif est de définir
les complexes descendant et montant de $\ua$. Il y a plusieurs manières
de procéder, chacune présentant un avantage. On peut par exemple définir
le descendant $\rmK_\sbullet(\ua)$ comme le produit tensoriel de complexes élémentaires
$$
\rmK_\sbullet(\ua) = \rmK_\sbullet(a_1) \otimes \cdots \otimes \rmK_\sbullet(a_n)
\overset{\rm can.}{\simeq}
\rmK_\sbullet(a_n) \otimes \rmK_\sbullet(\ua')
\qquad \text{avec} \quad
\ua' = (a_1, \cdots, a_{n-1})
$$
ce qui permet de relier les groupes d'homologie de $\rmK_\sbullet(\ua)$ et ceux
de $\rmK_\sbullet(\ua')$ à travers une longue suite exacte dans laquelle intervient la
multiplication par $a_n$.

\smallskip
Une autre manière de procéder consiste en une définition globale directe
dans laquelle interviennent le produit intérieur droit $\sbullet\intd\ua$ et le produit
extérieur $\ua \wedge \sbullet$.

C'est par cette dernière définition que nous commençons en considérant
une forme linéaire $\mu : L \to \bfA$ et en définissant le complexe
de Koszul descendant de $\mu$, le montant étant défini comme son dual.
Notre passerons ensuite rapidement au cas $\mu = [a_1, \cdots, a_n] : \bfA^n
\to \bfA$.

\subsubsection{Le descendant $\rmK_\sbullet(\mu)$}

\index{complexe de Koszul!descendant}%

Soit $\mu : L \rightarrow \bfA$ une forme linéaire.
Le complexe descendant  $\rmK_\sbullet(\mu)$ est le  complexe
$$
\xymatrix@M=0.4pc{
\cdots \ar[r] & 
\bigwedge^2 L \ar[r] & 
\bigwedge^1 L \simeq L \ar[r]^\mu & 
\bigwedge^0 L \simeq \bfA \ar[r] &
0}
$$
où  la différentielle est définie à partir du produit intérieur droit 
$\BW^\sbullet(L) \times \BW^{1}(L^\star) \rightarrow \BW^{\sbullet-1}(L)$
via $\partial \, = \,\sbullet \intd \mu$:
$$
\xymatrix @M=0.4pc{ 
\partial :\  \BW^d L  \ar[r] & 
\BW^{d-1} L 
}
\qquad \qquad 
\partial(x_1 \wedge \cdots \wedge x_d) \ = \
\sum_{i=1}^d \,(-1)^{i-1} \, \mu(x_i) \, 
x_1 \wedge \cdots \wedge \backslash \kern -6pt{x_i} \wedge \cdots \wedge x_d
$$
C'est l'unique anti-dérivation à gauche qui prolonge $\mu$.
Plus précisément, $\partial : \BW^{} L \rightarrow \BW^{} L$ est une application linéaire, 
de degré $-1$, caractérisée par le fait que $\partial(x) = \mu(x)$ pour $x$ de degré $1$ 
et par l'égalité suivante où $x$ et $y$ sont homogènes 
$$
\partial(x \wedge y) \ = \ 
\partial(x) \wedge y \ + \ (-1)^{\deg x}\, x \wedge \partial(y)
$$
On peut vérifier que $\partial \circ \partial = 0$ soit directement 
soit par récurrence en utilisant cette égalité, cf. le lemme~\ref{AntiDerivationCarreNul}.

\medskip
On étend les scalaires à un module $M$ en définissant 
le complexe $\rmK_\sbullet(\mu \,; M) = \rmK_\sbullet(\mu) \otimes_\bfA M$.
On a $\rmH_0(\mu) = \bfA/\Im \mu$ et $\rmH_0(\mu\,;M) = M/(\Im\mu\, M)$. 

\subsubsection{Le montant $\rmK^\sbullet(\mu)$}

\index{complexe de Koszul!montant}%

On définit le complexe de Koszul montant 
comme étant le dual du complexe descendant c'est-à-dire via  
$\rmK^\sbullet(\mu) = \Hom\big(\rmK_\sbullet(\mu), \, \bfA\big)$.
$$
\xymatrix@M=0.4pc{
0 \ar[r] & 
\bfA \simeq \left(\bigwedge^0 L \right)^\star \ar[r]^-{1\mapsto \mu} &
L^\star \simeq \left(\bigwedge^1 L \right)^\star \ar[r] & 
\left(\bigwedge^2 L \right)^\star \ar[r] & 
\cdots &
}
$$
\label{KoszulMontantExtScalaires}%
A coefficients dans un module $M$, il est défini
par $\rmK^\sbullet(\mu \, ; M) = \Hom\big(\rmK_\sbullet(\mu), \, M \big)$.
Signalons au passage que ce dernier complexe n'est pas égal à $\rmK^\sbullet(\mu) \otimes M$.
C'est cependant le cas si $L$ est libre de rang fini, car le morphisme canonique
$\Hom(K, \bfA) \otimes M \rightarrow \Hom(K, M)$
est un isomorphisme lorsque $K$ est libre de rang fini (ce que l'on peut
appliquer à $K = \BW^d L$).

\smallskip
Le module $\rmH^0(\mu)$ est l'idéal $\Ann\,\mu$.

\begin{prop}
Soit $\mu : L \rightarrow \bfA$ une forme linéaire et $\partial_\mu = \sbullet \intd \mu$ la différentielle
de son complexe de Koszul descendant.
En notant $\delta_x = x \wedge \sbullet$ pour $x \in L$, on a l'égalité 
$$
\partial_\mu \circ \delta_x \ + \ \delta_x \circ \partial_\mu 
\ = \ 
\mu(x)\, \Id
\qquad\qquad
\vcenter {
\xymatrix @R=1cm @C=1.2cm{
\rmK_{i+2} \ar[r]^{\partial_\mu} &\rmK_{i+1} \ar[r]^{\partial_\mu} & \rmK_{i}
\\ 
\rmK_{i+2} \ar[r]_{\partial_\mu} &\rmK_{i+1} \ar[lu]|{\textstyle\delta_x} 
\ar[u]|{\textstyle \times \mu(x)} \ar[r]_{\partial_\mu} & \ar[lu]|{\textstyle \delta_x} \rmK_{i}
\\
}
}
$$
En particulier, l'idéal $\Im \mu$ annule les modules de (co)homologie des complexes 
$\rmK_\sbullet(\mu \,; M)$ et $\rmK^\sbullet(\mu \,; M)$.
\end{prop}

\begin{proof}
Pour l'égalité, il suffit d'utiliser que $\partial_\mu$ est une anti-dérivation à gauche
qui prolongue $\mu$.
Autrement dit, pour $y \in \BW^{d} L$, on a 
$$
( \partial_\mu \circ \delta_x )(y) \ =\  \partial_\mu(x \wedge y) \ = \ 
\mu(x) \wedge y + (-1)^1 x \wedge \partial_\mu(y) \ = \ 
\mu(x) y -  (\delta_x \circ \partial_\mu)(y)
$$
Ainsi, pour $y \in \Ker \partial_\mu$, on a $\mu(x) y \in \Im \partial_\mu$.
Ceci montre que $\Im \mu$ annule $\rmH_\sbullet(\mu)$.
Il suffit ensuite de tensoriser par $\id_M$ ou bien de transposer l'égalité pour 
obtenir que $\Im \mu$ annule $\rmH_\sbullet(\mu \,; M)$ et $\rmH^\sbullet(\mu \,; M)$.
\end{proof}

\subsubsection{Les complexes de Koszul descendant et montant d'une suite $\ua = (a_1,\dots, a_n)$}

Ici on suppose que $L$ est un module libre de rang $n$ 
muni d'une base $\underline e = (e_1, \dots, e_n)$ de sorte que l'on peut confondre $L$ et $\bfA^n$.
La puissance $d$-ème extérieure de $\bfA^n$ est un $\bfA$-module libre que l'on munit 
de la base $(e_I)$ :
$$
\BW^d(\bfA^n)
\ = \
\bigoplus_{\# I = d} \ \bfA \, e_I
\qquad \qquad
\text{où \ $e_I = e_{i_1} \wedge \cdots \wedge e_{i_d}$ \ pour\  $I = (i_1 < \cdots < i_d)$}
$$

\subsubsection{Le descendant}

Pour une suite $\ua = (a_1,\dots, a_n)$ de $\bfA$, on définit son complexe de Koszul 
descendant $\rmK_\sbullet(\ua)$ comme étant $\rmK_\sbullet(\mu)$ où $\mu : \bfA^n \rightarrow \bfA$ 
est la forme linéaire dont la matrice dans la base $\underline e$ est $[a_1 \ \cdots \ a_n]$.
C'est donc un complexe de modules libres, de longueur $n+1$,
$$
0 \longrightarrow
\xymatrix@M=0.4pc @C=1.4cm{
\bigwedge^n (\bfA^n) \ar[r] & 
\quad \cdots\quad \ar[r] & 
\bigwedge^1 (\bfA^n) \simeq \bfA^n \ar[r]^{[a_1, \cdots, a_n]} & 
\bigwedge^0 (\bfA^n) \simeq \bfA 
}
\longrightarrow 0
$$
et dont la différentielle est $\partial = \sbullet \intd a^\star$ où $a^\star = \sum a_i e^\star_i$. 
En voici une description sur la base $\underline e$ :
$$
\partial(e_I) \ = \  e_I \intd a^\star \ = \
\bigoplus_{i \in I} \ 
(-1)^{\varepsilon_i(I)} \, a_i \, e_{I \setminus i}
$$
où $\varepsilon_i(I)$ désigne le nombre d'éléments de $I$ strictement inférieurs à $i$,
en accord avec l'égalité $e_I \intd e^\star_i = (-1)^{\varepsilon_i(I)}\,e_{I \setminus i}$ pour $i \in I$ (pour
$i\notin I$, le produit intérieur est nul).

\label{NOTA01-KoszulDown}%
\label{NOTA01-varepsiloniI}%

\smallskip
Le complexe $\rmK_\sbullet(\ua \,; M)$ est par définition $\rmK_\sbullet(\ua) \otimes_\bfA M$.
On a $\rmH_0(\ua\,;M) = M/(\ua\,M)$, en particulier $\rmH_0(\ua) = \bfA/\langle\ua\rangle$.

\label{NOTA01-HKoszulDown}%

\medskip

Voyons $\BW^d(\bfA^n)$ comme un $\bfA$-module libre de rang $\binom{n}{d}$ 
muni de la base $(e_I)_{\# I = d}$  indexée par les parties de $\{1,\dots, n\}$ ordonnée lexicographiquement.
Pour $n=4$, le complexe de Koszul descendant a l'allure suivante:
$$
\xymatrix@M=0.5pc{
0 \ar[r]
& \bfA^{1} \ar[r]^-{\partial_4}
& \bfA^{4} \ar[r]^-{\partial_3}
& \bfA^{6} \ar[r]^-{\partial_2}
& \bfA^{4} \ar[r]^-{\partial_1}
& \bfA^1 \ar[r]
& 0 }
$$
avec les différentielles 
$
\partial_1 =
\left[
\begin{array}{*{4}{c}}
a_1 & a_2 & a_3 & a_4 \\
\end{array}
\right]
$
et en notant $\overline a_i$ pour $-a_i$:
$$
\def\ov{\overline} 
\partial_4 =
\left[
\begin{array}{*{1}{c}}
\ov a_4 \\
a_3 \\
\ov a_2 \\
a_1 \\
\end{array}
\right]
\qquad
\partial_3 =
\left[
\begin{array}{*{4}{c}}
a_3 & a_4 & . & . \\
\ov a_2 & . & a_4 & . \\
. & \ov a_2 & \ov a_3 & . \\
a_1 & . & . & a_4 \\
. & a_1 & . & \ov a_3 \\
. & . & a_1 & a_2 \\
\end{array}
\right]
 \qquad
\partial_2 =
\left[
\begin{array}{*{6}{c}}
\ov a_2 & \ov a_3 & \ov a_4 & . & . & . \\
a_1 & . & . & \ov a_3 & \ov a_4 & . \\
. & a_1 & . & a_2 & . & \ov a_4 \\
. & . & a_1 & . & a_2 & a_3 \\
\end{array}
\right]
$$

\subsubsection{Le montant}

Pour une suite $\ua = (a_1,\dots, a_n)$ de $\bfA$, on définit son
complexe de Koszul montant $\rmK^\sbullet(\ua)$ comme étant
$\rmK^\sbullet(\mu)$ où $\mu : \bfA^n \rightarrow \bfA$ est la forme
linéaire (notée $a^\star$ ci-après) dont la matrice dans la base
canonique est $[a_1 \ \cdots \ a_n]$.  Grâce à
l'isomorphisme \emph{canonique} $\big(\BW^d
(\bfA^n) \big)^\star \simeq \BW^d \big( (\bfA^n)^\star \big)$ et
au fait que $a^\star \wedge \sbullet$ et $\sbullet\intd a^\star$ sont adjoints
(cf la section \ref{ExteriorAlgebraDuality}), le complexe de Koszul
prend la forme suivante :
$$
\xymatrix@M=0.4pc{
0 \ar[r] &
\bigwedge^0 \big((\bfA^n)^\star \big) \simeq \bfA \ar[r]^-{1\mapsto a^\star}  &
\bigwedge^1  \big((\bfA^n)^\star \big) \simeq (\bfA^n)^\star \ar[r] &  
\quad \cdots\quad \ar[r] & 
\bigwedge^n  \big((\bfA^n)^\star \big) \ar[r] &
0
}
$$
avec la différentielle $\delta = a^\star \wedge \sbullet$\, qui
s'explicite sur la base duale $(e_I^\star)$, où $e_I^\star =
e_{i_1}^\star \wedge \cdots \wedge e_{i_d}^\star$ pour $I = \{i_1 <
\cdots < i_d\}$, de la façon suivante
$$
\delta(e_I^\star)  = a^\star \wedge e_I^\star =
\ = \ 
\bigoplus_{j \notin I}\ 
(-1)^{\varepsilon_j(I)} \, a_j \, e_{j\vee I}^\star
$$
On rappelle que $\varepsilon_j(I)$ désigne le nombre d'éléments de $I$ strictement inférieurs à $j$,
en accord avec $e^\star_j \wedge e^\star_I  = (-1)^{\varepsilon_j(I)}\,e^\star_{\{j\}\vee I}$ pour $j \notin I$
(pour $j \in I$, le produit extérieur est nul).

\smallskip
Le complexe $\rmK^\sbullet(\ua \,; M)$ est par définition $\Hom(\rmK_\sbullet(\ua),\, M)$.
Comme $\rmK_\sbullet(\ua)$ est un complexe de modules libres, 
on a $\rmK^\sbullet(\ua \,; M)  = \rmK^\sbullet(\ua) \otimes_\bfA M$ 
(confer ce qui a été dit à la page~\pageref{KoszulMontantExtScalaires}).
On a $\rmH^0(\ua) = \Ann(\ua)$,  $\rmH^0(\ua\,;M) = \{m \in M \mid \ua\,m = 0\}$
et  $\rmH^n(\ua\,;M) = M/(\langle\ua\rangle\,M)$.

\label{NOTA01-KoszulUp}%
\label{NOTA01-HKoszulUp}%
%
%

\subsubsection{Auto-dualité du complexe de Koszul}

\index{auto-dualité du complexe de Koszul}%

La conjugaison entre $\sbullet \intd a^\star$ et $a^\star \wedge\sbullet$ 
étudiée à la page~\pageref{AutoDuality}  permet d'énoncer:

\begin{prop} \label{KoszulAutoDualite}
Les complexes de Koszul $\rmK_\sbullet(\ua)$ et $\rmK^\sbullet(\ua)$ 
sont isomorphes. Autrement dit, le complexe~\mbox{$\rmK_\sbullet(\ua)$} est isomorphe 
à son dual.
Plus généralement, les complexes de Koszul $\rmK_\sbullet(\ua\,; M)$ et $\rmK^\sbullet(\ua \,; M)$ 
sont isomorphes.

\noindent
En conséquence, pour tout entier $i$,
$$
\rmH_i(\ua \,; M) \ \simeq \ \rmH^{n-i}(\ua \,; M)
$$
\end{prop}

Ci-dessous sont superposés le complexe de Koszul descendant et le complexe 
de Koszul montant pour $n=4$ dans lesquels on note $\overline a_i$ pour $-a_i$.
$$
\def\ov{\overline} 
\BW^{4} \xrightarrow{%
\begin{bmatrix}
\ov a_4 \\
a_3 \\
\ov a_2 \\
a_1 \\
\end{bmatrix}
}
\BW^{3} \xrightarrow{%
\begin{bmatrix}
a_3 & a_4 & . & . \\
\ov a_2 & . & a_4 & . \\
. & \ov a_2 & \ov a_3 & . \\
a_1 & . & . & a_4 \\
. & a_1 & . & \ov a_3 \\
. & . & a_1 & a_2 \\
\end{bmatrix}
}
\BW^{2} \xrightarrow{%
\begin{bmatrix}
\ov a_2 & \ov a_3 & \ov a_4 & . & . & . \\
a_1 & . & . & \ov a_3 & \ov a_4 & . \\
. & a_1 & . & a_2 & . & \ov a_4 \\
. & . & a_1 & . & a_2 & a_3 \\
\end{bmatrix}
}
\BW^{1} \xrightarrow{%
\begin{bmatrix}
a_1 & a_2 & a_3 & a_4 \\
\end{bmatrix}
}
\BW^{0}
$$
$$
\def\ov{\overline} 
\BW^{0} \xrightarrow{%
\begin{bmatrix}
a_1 \\ a_2 \\ a_3 \\ a_4 
\end{bmatrix}
}
\BW^{1} \xrightarrow{%
\begin{bmatrix}
\ov a_2 & a_1 & . & . \\
\ov a_3 & . & a_1 & . \\
\ov a_4 & . & . & a_1\\
. & \ov a_3 & a_2 & . \\
. & \ov a_4 & . & a_2 \\
. & . & \ov a_4 & a_3 \\
\end{bmatrix}
}
\BW^{2} \xrightarrow{%
\begin{bmatrix}
a_3 & \ov a_2 & . & a_1 & .& . \\
a_4 & . & \ov a_2 & . & a_1 & . \\
. & a_4 & \ov a_3 & . & . &  a_1 \\
. & . & . & a_4 & \ov a_3 &  a_2 \\
\end{bmatrix}
}
\BW^{3} \xrightarrow{%
\begin{bmatrix}
\ov a_4 &
a_3 &
\ov a_2 &
a_1 &
\end{bmatrix}
}
\BW^{4}
$$
\subsubsection{Le complexe de Koszul comme produit tensoriel de complexes élémentaires}

\noindent
Les complexes qui interviennent ci-dessous sont montants.
Le produit tensoriel $C \otimes C'$ de deux complexes
$(C, \delta)$ et $(C', \delta')$ est le complexe $(C'',\delta'')$ d\'efini
par
$$
C''_i = \bigoplus_{j+k=i} C_j \otimes C'_k, \qquad
\delta'' : x \otimes x' \mapsto \delta x \otimes x' + (-1)^{\deg x} x \otimes \delta'x'
$$
Pour $t \in \bbZ$, le complexe d\'ecal\'e $C[t]$ de $C$ est le complexe d\'efini par
la venerable formulae $C[t]_n = C_{t+n}$, sa diff\'erentielle
\'etant $(-1)^t \delta_C$.

\index{produit tensoriel!de complexes}%

\noindent
Pour $b \in \bfA$, le complexe de Koszul montant $\rmK^\sbullet(b)$ de
la suite r\'eduite \`a $b$ est parfois qualifié d'élémentaire car
concentré en degrés $0$ et $1$:
$$
\rmK^\sbullet(b) : \bfA \xrightarrow{\times b} \bfA \to 0 \to 0 \to \cdots
\qquad\qquad
\rmK^\sbullet(b\,; M) : M \xrightarrow {\times b} M \to 0 \to 0 \to \cdots
$$

\`A l'aide de la construction mapping-cone et de
l'identification $\rmK(\ua, b \,;
M) \simeq \rmK(b \,; \bfA) \otimes_\bfA \rmK(\ua \,; M)$, on obtient
facilement les longues suites exactes en homologie et cohomologie de
Koszul qui relient l'homologie de $(\ua;b)$ et la multiplication par
$b$ sur l'homologie de $\ua = (a_1, \ldots, a_{n})$.  Nous énonçons
le théorème suivant sans preuve: pour les détails, confer
l'appendice~A de \cite{Tete}.

\index{produit tensoriel!par un complexe élémentaire}%

\begin {theo} [Tensorisation par un complexe élémentaire]
\leavevmode

Soit $(C,\delta)$ un complexe (montant) et~$b \in \bfA$

\begin {enumerate} [\rm i)]
\item
Le produit tensoriel $\rmK^\sbullet(b) \otimes C$ est isomorphe au 
c\^one de $C[-1] \xrightarrow {\times b} C[-1]$, celui ci \'etant
le complexe $C \oplus C[-1]$ avec 
bord $x_i \oplus x_{i-1} \mapsto \delta x_i \oplus (bx_i - \delta x_{i-1})$.

\item
On dispose d'une suite exacte courte:
$$
0 \to C[-1] \xrightarrow {x' \mapsto 0\oplus x'} C \oplus C[-1]
\xrightarrow {x \oplus x' \mapsto x} C \to 0
$$

\item
Cette suite exacte courte donne naissance à la longue suite exacte en cohomologie:
$$
\cdots \to \rmH^{i-1}(C) \to \rmH^{i}(\rmK^\sbullet(b) \otimes C) \to 
\rmH^{i}(C) \xrightarrow {\times b} \rmH^{i}(C) \to 
\rmH^{i+1}(\rmK^\sbullet(b) \otimes C) \to \cdots
$$

\item
On dispose d'une identification canonique $\rmK^\sbullet(\ua, b\,; M) \simeq 
\rmK^\sbullet(b \,; \bfA) \otimes_\bfA \rmK^\sbullet(\ua \,; M)$.
\end {enumerate}
\end {theo}

\subsubsection{La longue suite exacte en (co)homologie Koszul}

Ces deux longues suites exactent découlent du résultat précédent.

\begin{prop}[Longue suite en (co)-homologie Koszul]
Pour  une suite $\ua$ d'éléments de $\bfA$ et $b \in \bfA$, on dispose
de deux longues suites exactes:

\noindent
En homologie:
$$
\xymatrix{
\cdots \ar[r] & 
\rmH_{i+1}(\ua \,; M) \ar[r] & \rmH_{i+1}(\ua, b \,; M) \ar[r] & \rmH_i(\ua \,; M) 
\ar `r[d] `[l] `[llld]_{\times b} `[dll] [dll] \\
& \rmH_i(\ua \,; M) \ar[r] & \rmH_{i}(\ua, b \,; M) \ar[r] & \rmH_{i-1}(\ua \,; M) \ar[r] & \cdots
}
$$
En cohomologie:
$$
\xymatrix{
\cdots \ar[r] & \rmH^{i-1}(\ua \,; M) \ar[r] & \rmH^{i}(\ua, b \,; M) \ar[r] & \rmH^i(\ua \,; M) 
\ar `r[d] `[l] `[llld]_{\times b} `[dll] [dll] \\
& \rmH^i(\ua \,; M) \ar[r] & \rmH^{i+1}(\ua, b \,; M) \ar[r] & \rmH^{i+1}(\ua \,; M) \ar[r] & \cdots}
$$
\end{prop}

La proposition suivante va nous permettre de démontrer l'invariance par radical de la profondeur.

\begin{prop}\label{ConsequenceLongueSuiteCohomologieKoszul}
Soit $\ua$ une suite et $b \in \bfA$. 
Soit $i$ un entier.
\begin{enumerate} [\rm i)]
\item 
Si $\rmH^{i-1}(\ua \,; M) = \rmH^i(\ua \,; M) = 0$ alors $\rmH^i(\ua, b \,; M) = 0$.
\item 
Soit $e$ un entier tel que $b^e \in \langle \ua \rangle$.
Si $\rmH^i(\ua, b \,; M) = 0$ alors $\rmH^i(\ua \,; M) = 0$.
\end{enumerate}
On obtient les mêmes résultats en remplaçant l'élément $b$ par une suite $\ub$.
\end{prop}

\begin{proof}
Le point i) est immédiat en utilisant la longue suite exacte en cohomologie Koszul.
Pour le point~ii), on reprend également cette longue suite.
Comme $\rmH^i(\ua, b \,; M) = 0$, 
la flèche \mbox{$\rmH^i(\ua \,; M) \xrightarrow{\times b} \rmH^i(\ua \,; M)$}
est injective. Il en est de même de la flèche de multiplication par $b^e$.
En utilisant que $b^e \in \langle \ua \rangle$ et le fait que la cohomologie 
Koszul est tuée par $\langle \ua \rangle$, on obtient $\rmH^i(\ua \,; M) = 0$.
\end{proof}

\begin{prop}\label{InvarianceRadicalGrade}
Soient $\ua$ et $\ub$ deux suites d'éléments de $\bfA$.
On suppose qu'il existe un entier $e$ tel que $\ub^e \in \langle \ua \rangle$
(ce qui est équivalent à $\sqrt{\mathstrut{\ub}} \subset \sqrt{\mathstrut{\ua}}$).
Pour tout entier $k$, on a 
$$
\rmH^i(\ub \,; M) = 0, \quad \forall\, i < k 
\qquad \Longrightarrow \qquad
\rmH^i(\ua \,; M) = 0, \quad \forall\, i < k 
$$
En particulier, si $\sqrt{\mathstrut{\ua}} = \sqrt{\mathstrut{\ub}}$ alors 
$$
\rmH^i(\ua \,; M) = 0, \quad \forall\, i < k 
\qquad \Longleftrightarrow \qquad
\rmH^i(\ub \,; M) = 0, \quad \forall\, i < k 
$$
\end{prop}

\label{NOTA01-Sqrtua}%

\begin{proof}
Supposons que $\rmH^i(\ub \,; M) = 0$ pour tout $i < k$.
Le premier point de la proposition~\ref{ConsequenceLongueSuiteCohomologieKoszul}
permet d'obtenir (en échangeant les rôles de $\ua$ et $\ub$ dans la proposition)
$\rmH^i(\ub, \ua \,; M) = 0$ pour tout $i<k$.
Et le deuxième point fournit $\rmH^i(\ua \,; M) = 0$ pour tout $i<k$.
\end{proof}

\begin{defns}
\leavevmode  

\begin {enumerate}  [\rm i)]
\item 
Soit $k$ un entier.
On définit pour une suite $\ua = (a_1,\dots, a_n)$ 
et un module $M$ l'assertion $\Gr(\ua \,; M) \geqslant k$
par l'annulation des $k$ premiers modules de \emph{cohomologie} Koszul.

\setlength{\fboxsep}{2mm}
\begin{center}
\fbox{$
\Gr(\ua \,; M) \geqslant k 
\quad \Longleftrightarrow \quad 
\forall \, i < k, \quad \rmH^i(\ua \,; M) = 0$
}
\end{center}

\item
Une suite $\ua = (a_1, \dots, a_n)$ est dite $M$-complètement sécante si $\Gr(\ua \,; M) \geqslant n$
 \idest{} $\rmH^i(\ua\,;M) = 0$ pour $0 \le i < n$. 
\end{enumerate}
\end{defns}

\label{NOTA01-GrSuite}
\index{suite!complètement sécante}%

\noindent
Les propriétés suivantes de la profondeur se déduisent des propositions précédentes.

\begin{coro}\label{GrCompatibleInclusionRacine}
Soit $k$ un entier.
\begin{enumerate} [\rm i)]

\item 
Soit $b \in \bfA$. Alors $\Gr(\ua \,; M) \geqslant k \ \Rightarrow \ \Gr(\ua, b \,; M)  \geqslant k$

\item 
Si $\sqrt{\mathstrut \ub} \, \subset \sqrt{\mathstrut \ua}$ alors 
$\Gr(\ub \,; M) \geqslant k \ \Rightarrow \ \Gr(\ua \,; M) \geqslant k$.

\item
Si $\sqrt{\mathstrut \ua} \, =  \sqrt{\mathstrut \ub}$ alors $\Gr(\ua \,; M) \, = \, \Gr(\ub \,; M)$.

\item
Pour un idéal $\fa$ de type fini, on peut définir $\Gr(\fa\,;M)$ comme étant $\Gr(\ua\,;M)$ pour
n'importe quel système générateur fini $\ua$ de $\fa$.
\end{enumerate}
\end{coro}

\label{NOTA01-GrIdeal}
%
%

\index{profondeur!d'une suite de scalaires}%
\index{profondeur!d'un idéal de type fini}%

\begin {coro} \label{ReguliereImpliqueCompletementSecant}
\leavevmode
  
\begin {enumerate} [\rm i)]
\item
Si $\ua = (a_1, \cdots, a_n)$ est $M$-régulière, alors elle est $M$-complètement sécante.
\item
Si un idéal de type fini $\fa$ contient une suite $M$-régulière de longueur $n$, alors $\Gr(\fa\,;M) \ge n$.
\end {enumerate}  
\end {coro}  

\begin {proof} \leavevmode

i)   
On raisonne par récurrence sur $n$. La suite $\ua' = (a_1, \ldots, a_{n-1})$
est $M$-régulière donc $\rmH^i(\ua'\,; M) = 0$ pour $i < n-1$. Pour un tel $i$,
$\rmH^i(\ua\,;M)$ est encadré dans la suite exacte longue par
$\rmH^{i-1}(\ua'\,; M)$ et $\rmH^{i}(\ua'\,;M)$ qui sont nuls, donc
$\rmH^i(\ua,M) = 0$. Reste \`a traiter $\rmH^{n-1}(\ua, M)$; on
dispose d'une suite exacte:
$$
0 = \rmH^{n-2}(\ua'\,;M) \longrightarrow \rmH^{n-1}(\ua\,;M) \longrightarrow 
\rmH^{n-1}(\ua'\,;M) \xrightarrow {\times a_n} \rmH^{n-1}(\ua'\,; M) 
$$
De plus $\rmH^{n-1}(\ua'\,; M) = M/\ua'M$ donc $\rmH^{n-1}(\ua\,; M)$ 
est isomorphe au noyau de la multiplication par $a_n$ sur
$M/\ua'M$; par définition d'une suite régulière, ce noyau est nul.

\medskip
ii) Soit $\ub$ un système fini de générateurs de $\fa$ et $\ua$ une suite régulière contenue
dans $\fa$ de longueur $n$. Alors $\Gr(\fa\,;M) = \Gr(\ub\,;M) \ge \Gr(\ua\,;M) \ge n$.
\end {proof}

\medskip

Nous citons sans preuve le résultat important suivant qui n'est pas aisé à établir via
le complexe de Koszul. Il fait l'objet de la proposition 6.6 de~\cite{CoquandQuitte}.
Nous le prendrons en charge pour $k=2$ dans la section suivante.

\begin {theo}
Soient $\ua, \ub$ deux suites finies de scalaires. Si $\Gr(\ua\,;M) \ge k$ et $\Gr(\ub\,;M) \ge k$,
alors $\Gr(\ua\,\ub\,;M) \ge k$.
\end {theo}

\subsection{Profondeur $\geqslant 2$ et pgcd fort}

Vu l'importance dans notre étude de la notion de profondeur (supérieure ou égale à)~2,
il nous a paru important d'en assurer les fondements de la manière la plus
élémentaire possible, en évitant de dire \og $\Gr \geqslant 2$, c'est $\Gr \geqslant k$ pour
$k = 2$\fg{}. De ce fait, nous reprenons en charge certains résultats liés à la profondeur~2.

\medskip
\noindent
Soit $\ua = (a_1, \ldots, a_m)$ une suite de scalaires.  La
définition homologique de la propriété 
\og $\ua$ est de profondeur $\geqslant 2$\fg{}, consiste à considérer le complexe de Koszul montant
$\rmK^\sbullet(\ua)$, et à dire que les 2 premiers modules de
cohomologie Koszul $\rmH^0(\ua)$ et $\rmH^1(\ua)$ sont nuls. 
Rappelons les deux premières différentielles~:
$$
\xymatrix @C=2cm{
\rmK^0(\ua) = \bfA \ar[r]^-{\begin {bmatrix} a_1\cr \vdots\cr a_m\end{bmatrix}} &
\rmK^1(\ua) = \bigoplus_{j=1}^m \bfA e_j \ar[r]^-{e_j \mapsto \sum_i a_i\, e_i\wedge e_j} &
\rmK^2(\ua) = \bigoplus_{i<j}\bfA e_i\wedge e_j\ \mapsto \cdots 
}
$$
Pour $y = \sum_j y_j e_j \in \rmK^1(\ua)$, la différentielle $y' := \delta_1(y) = \big(\sum_i a_ie_i\big) \wedge y$
est donc donnée par:
$$
y' = \sum_{i < j} y'_{i,j}\, e_i\wedge e_j \quad\hbox{avec}\quad y'_{i,j} = a_i y_j - a_j y_i
$$
Ceci légitime la définition directe suivante, sans allusion franche au complexe de Koszul.

\begin{defn} [Profondeur $\geqslant 2$]
Une famille finie de scalaires $\ua = (a_i)$ est dite de profondeur~{$\geqslant 2$}
si d'une part $\Ann(\ua) = 0$ et si pour toute suite
de scalaires $(y_i)$ vérifiant $a_iy_j = a_jy_i$, il existe un $q$
tel que $a_i = qy_i$ pour tout~$i$. Un tel $q$ est unique puisque
$\Ann(\ua) = 0$.
\end{defn}

\label{NOTA01-Gr2}%

On remarque que cette définition ne fait pas intervenir l'ordre des éléments.
Il est donc impératif d'introduire rapidement l'idéal engendré. 

\begin {prop} \label{Gr2ProprietesElementaires}
\leavevmode

\begin{enumerate}[\rm i)]
\item 
Si $\langle\ua\rangle \subset \langle\ua'\rangle$, alors
$\Gr(\ua) \geqslant  2 \Rightarrow \Gr(\ua') \geqslant  2$.

\item 
La propriété $\Gr(\ua) \geqslant  2$ ne dépend que de l'idéal engendré par $\ua$.

\item 
Une suite régulière de longueur $2$ est de profondeur $\geqslant 2$ ; 
en conséquence, si $\langle\ua\rangle$ contient une suite régulière de
longueur $2$, alors $\Gr(\ua) \geqslant 2$.
\end{enumerate}
\end{prop}

\begin {proof} 
L'idée est de faire reposer i) sur les deux propriétés suivantes
où $b',c$ sont des scalaires:  
$$
\Gr(\ua) \geqslant 2 
\ \Rightarrow\ 
\Gr(\ua, b') \geqslant  2 
\qquad \qquad \text{et }\qquad \qquad 
\big(\Gr(c,\ua') \geqslant  2 \hbox { et } c \in \langle\ua'\rangle\big)
\ \Rightarrow\  \Gr(\ua') \geqslant  2
\leqno (\star)
$$
Nous en laissons la vérification au lecteur.

Pour en déduire le point i), on écrit $\ua = (a_1, \ldots, a_m)$. 
On a alors $\Gr(a_1,\ldots,a_m,\ua') \geqslant  2$, d'après la propriété de gauche de $(\star)$. On
peut ensuite, en utilisant la propriété de droite de manière itérée, 
supprimer un à un chaque~$a_i$.

ii) Si on suppose $\langle\ua\rangle = \langle\ua'\rangle$, alors d'après le point i), on a 
$\Gr(\ua) \geqslant  2 \iff \Gr(\ua') \geqslant 2$.

iii) 
Soit $\ua = (a_1, a_2)$ régulière. Comme $a_1$ est régulier, on a $\Ann(\ua) = 0$.
Concernant la nullité du~$\rmH^1(\ua)$, considérons 
$(x_1,x_2)$ vérifiant $a_1x_2 = a_2x_1$. 
On a alors $a_2x_1 \equiv 0 \bmod a_1$ 
d'où $x_1 \equiv 0 \bmod a_1$, 
donc il existe~$q$ tel que $x_1 = qa_1$. 
En reportant cela dans l'égalité initiale, 
on obtient $a_1x_2 = a_2qa_1$. 
En simplifiant par~$a_1$, on a $x_2 = qa_2$. 
Bilan : $(x_1,x_2) = q(a_1,a_2)$.
\end {proof}

\medskip

Les notions \og $\Gr(\ua) \geqslant 2$ \fg{} et 
\og $\langle\ua\rangle$ contient une suite régulière de longueur~2 \fg{} ne
doivent pas être confondues : la première est stable par permutation
des termes de la suite, mais pas la seconde en général.

\medskip

On rappelle que la notion de pgcd de monômes est une notion uniquement
liée au monoïde ordonné des exposants, de ce fait indépendante de
l'anneau des coefficients. Par définition, pour deux monômes $m,m'$:
$$
m = T^\alpha, \quad m' = T^{\alpha'}, \qquad
\pgcd(m,m') = T^{\inf(\alpha,\alpha')}
$$
de sorte que:
\begin{center}
\it
Pour $m, m_1, m_2,T^\alpha$ quatre monômes 
de $\bfk[\uT]$, si $m \mid m_1 T^\alpha$ et $m \mid m_2 T^\alpha$, 
alors~{$m \mid \pgcd(m_1,m_2)T^\alpha$.}
\end{center}

\noindent
On en déduit la propriété

\begin{quote}
\it
Soit $F \in \bfk[\uT]$ et $m, m_1, m_2$ trois monômes
de $\bfk[\uT]$ vérifiant $m \mid m_1F$ et $m \mid m_2F$. 

Alors $m \mid \pgcd(m_1,m_2)F$.

En particulier si $m \mid m_1F$ et $\pgcd(m,m_1) = 1$, alors $m \mid F$ (prendre $m_2 = m$).
\end{quote}

\noindent
\'Ecrivons $F = \sum_\alpha a_\alpha T^\alpha$ en nous limitant aux indices $\alpha$ tels que
$a_\alpha \ne 0$. 
Puisque $m \mid m_1F = \sum_\alpha a_\alpha m_1 T^\alpha$, 
on a $m \mid m_1 T^\alpha$ pour chaque $\alpha$. 
De la même façon $m \mid m_2 T^\alpha$. 
Donc $m \mid \pgcd(m_1,m_2) T^\alpha$. 
En définitive, $m \mid \pgcd(m_1,m_2)F$.

\begin {prop}[Monômes premiers entre eux]
\label{MonomesPremiersEntreEux}
Dans un anneau de polynômes $\bfk[\uT]$ sur un anneau commutatif $\bfk$,
considérons $s$ monômes $m_1, \dots, m_s$. On a alors l'équivalence
$$
\hbox {$m_1, \dots, m_s$ premiers dans leur ensemble}  \quad \Longleftrightarrow \quad
\Gr(m_1, \dots, m_s) \geqslant 2
$$
\end {prop}

\begin {proof} \leavevmode

\noindent 
Implication $\Rightarrow$.
Soient $F_1, \dots, F_s \in \bfk[\uT]$ vérifiant $F_i m_j = F_j
m_i$. Il nous faut montrer qu'il existe un (unique) $G$ tel que $F_i =
Gm_i$. Le monôme $m_1$ divise chaque $m_2F_1, \dots, m_sF_1$ donc,
par itération de la propriété figurant en préambule, en posant $m'_1 = \pgcd(m_2, \dots,
m_s) F_1$, on a $m_1 \mid m'_1 F_1$. Et comme $\pgcd(m_1,m'_1) = 1$, on
en déduit $m_1 \mid F_1$.

Ce que l'on a fait avec l'indice 1, on peut le faire pour n'importe quel indice $i$ pour obtenir
$m_i \mid F_i$, disons $F_i = G_im_i$.  En considérant de nouveau $F_i m_j = F_j m_i$, on
obtient $G_im_im_j = G_jm_jm_i$ donc $G_i = G_j$. Et l'unique $G$ qui convient est la valeur
commune des $G_i$.

\medskip
\noindent
Implication $\Leftarrow$. Soit $q$ le monôme $\pgcd(m_1, \dots, m_s)$, donc $m_i = qm'_i$
avec $m'_i$ monôme. On a $m'_im_j = m'_jm_i$. Comme $\pgcd(\um') = 1$, on 
a $\Gr(\um') \geqslant 2$,  donc il existe $G$ tel que $m'_i = Gm_i$, ce qui conduit à $Gq = 1$
puis à $q= 1$.
\end {proof}


\medskip
\noindent
La propriété $\Gr(\ua) \geqslant 2$ peut être intuitivement perçue
comme le fait que les éléments de $\ua$ sont premiers dans leur
ensemble dans un sens très fort.  Par exemple, comme nous venons de le
voir, une famille finie de monômes sur un anneau quelconque est de
profondeur $\geqslant 2$ si et seulement si ils sont premiers dans leur
ensemble.  Dans un anneau intègre à pgcd (par exemple un anneau
factoriel), une famille d'éléments est de profondeur~$\geqslant 2$ si
et seulement s'ils sont premiers dans leur ensemble. Le lemme qui suit
formalise cette terminologie ``premiers entre eux dans un sens fort''.

\begin{lem}[Neutralisation des idéaux de profondeur $\geqslant 2$] 
\label{Gr2StrongGcd} 

\leavevmode

Soit $\ub$ une suite de scalaires de profondeur $\geqslant  2$.

\begin{enumerate}[\rm i)]
\item 
Soient $h,g \in \bfA$ avec $h$ régulier vérifiant $h \mid gb_i$ pour tout $i$;
alors $h \mid g$. 

\item 
Soit $h \in \bfA$ tel que $h \mid b_i$ pour tout $i$; alors
$h$ est inversible.
\end{enumerate}
\end {lem}

\begin {proof}\leavevmode

i) Par hypothèse, il existe $q_i \in \bfA$ tel que $gb_i = q_ih$; on en déduit
$q_ib_j h = q_jb_i h \ (= gb_ib_j)$ et puisque $h$ est régulier,
$q_ib_j = q_jb_i$. Comme $\Gr(\ub) \geqslant  2$, il existe $q \in \bfA$
tel que $q_i = qb_i$. On reporte cette valeur de $q_i$ dans l'égalité
initiale $gb_i = q_ih$ pour obtenir $gb_i = qb_ih$ \idest{} $(g-qh) b_i = 0$
pour tout $i$.
Comme $\Ann(\ub) = 0$, on~a $g = qh$ et on a bien montré que $h \mid g$.

ii) Puisque $h \mid b_i$ pour tout $i$ et $\Ann(\ub) = 0$, 
l'élément $h$ est régulier et on peut appliquer le point i) avec $g=1$.
\end{proof}  

\begin{rmq} \label{NeutralisationGr2}
La propriété i) du lemme précédent peut être vue comme une
\og neutralisation des idéaux de profondeur $\geqslant  2$\fg{} et être écrite de la
manière suivante : 
\begin{center}
\textit{Soit $\fb$ un idéal de type fini vérifiant $\Gr(\fb) \geqslant  2$;
alors pour $h$ régulier :} 
$$
\forall\, g \in \bfA, \quad 
\langle g\rangle\,\fb \subset \langle h\rangle   \quad\Rightarrow\quad
\langle g\rangle \subset \langle h\rangle
$$
\end{center}
\end{rmq}

\medskip
Dans un anneau commutatif quelconque, il est possible de décréter
qu'un élément régulier $d$ est le pgcd d'une famille finie d'éléments
(d'annulateur réduit à 0) si d'une part $d$ est un diviseur commun et
si d'autre part tout diviseur commun est un diviseur de~$d$.  C'est
une définition ponctuelle faible, d'un maigre intérêt. En effet, 
si $d$ est un pgcd de $\ua$ en ce sens et $c$ un élément
régulier, il n'y a aucune raison pour que le pgcd de $c\ua$ soit $cd$
(cf l'exemple~\ref{ContreExemplePgcdFort} à venir) ;
une telle propriété de divisibilité est
pourtant le ``minimum vital requis''.

\begin{exemple} \label{ContreExemplePgcdFort}
Considérons la suite $\ua = (t^2,t^3)$ dans l'anneau $\bfA = \bbZ[t^2,t^3]$.

$\rhd$ 
En utilisant le fait que $t \notin \bfA$, on voit que les éléments
$t^2, t^3$ sont premiers entre eux (au sens où les seuls
diviseurs communs sont les inversibles).  

$\rhd$
Cependant, on n'a pas
$\Gr(\ua) \geqslant  2$; en effet, les suites $\ua = (t^2,t^3)$ et $\ub = (t^3, t^4)$
sont proportionnelles (au sens où $a_1b_2 = a_2b_1$) mais $\ub$ n'est
pas multiple de $\ua$.

$\rhd$
Comme $1$ est un pgcd faible de $\ua$, on pourrait s'attendre à ce que 
$t^3$ soit un pgcd de $t^3\ua$.
Mais il n'en est rien
car $t^3\ua = (t^5,t^6)$ admet $t^2$ comme diviseur commun, et $t^2$ ne divise pas $t^3$
dans $\bfA$.

$\rhd$
Introduisons $\ub = (t^3, t^4)$. Il est clair que $\ub$ a pour pgcd~1. 
Quid de $\ua\,\ub$ ? On a~:
$$
\langle \ua\,\ub\rangle = \langle t^5, t^6, t^6, t^7\rangle =
\langle t^5, t^6\rangle
$$
Mais la suite $(t^5, t^6)$ n'a pas pour pgcd~1 car $t^2$ et $t^3$ sont
des diviseurs communs des termes de cette suite.
Pire: cette suite ne possède pas de pgcd!
\end{exemple}

Au regard du contre-exemple précédent, nous souhaitons mettre place
une notion très robuste pour palier ce type de désagrément.  Cela
conduit à la définition suivante :

\begin{defn}[Pgcd fort d'un idéal de type fini] 
On dit qu'un scalaire $g$ est un pgcd fort d'un idéal de type fini $\fa$
si $g$ est régulier, s'il divise tout élément de $\fa$ et
vérifie $\Gr(\fa/g) \geqslant 2$.  
Ce qui revient à demander l'existence d'une factorisation du type 
$$
\fa = \langle g \rangle \fb 
\qquad 
\text{avec } \qquad 
\left\{
\begin{array}{l}
g \text{ régulier} \\ [1.5mm]
\Gr(\fb) \geqslant 2
\end{array}
\right.
$$
Un pgcd fort, s'il existe, est unique à un inversible près. 
L'idéal $\fb$ est unique.
\end{defn}

\index{pgcd fort}%

\begin{proof}[Preuve de l'unicité]
Si $\fa = \langle g\rangle \fb = \langle g' \rangle \fb'$, 
alors $\langle g\rangle \fb \subset \langle g' \rangle$
et d'après le principe de neutralisation des idéaux de profondeur $\geqslant 2$, 
on obtient 
$\langle g\rangle \subset \langle g'\rangle$. 
Idem en échangeant les rôles de $g$ et $g'$.
On en déduit que $g$ et $g'$ sont associés puis que $\fb = \fb'$.
\end{proof}

\begin{prop}\label{InclusionPGCD}
Soit $\fa$ et $\fa'$ deux idéaux de type fini admettant un pgcd fort, disons $g$ et $g'$.

Si $\fa \subset \fa'$, alors $\langle g \rangle \subset \langle g' \rangle$.
\end{prop}

\begin{proof}
On a $\fa = \langle g \rangle \fb$ avec $\Gr(\fb) \geqslant 2$. 
L'hypothèse $\fa \subset \fa' \subset \langle g' \rangle$ 
implique $\langle g \rangle \fb \subset \langle g' \rangle$.
D'après le principe de neutralisation des idéaux de profondeur $\geqslant 2$, 
on obtient $\langle g\rangle \subset \langle g'\rangle$. 
\end{proof}

\begin{defn}[Pgcd fort d'un sous-module de type fini dans le module ambiant] \label{PgcdFortModule}
Soit $F$ un sous-module de type fini d'un module $E$ et $\vartheta \in E$.
On dit que $\vartheta$ est un pgcd fort de $F$ dans $E$ si $\vartheta$ 
est sans torsion, s'il divise tout élément de~$F$ et
vérifie $\Gr(F/\vartheta) \geqslant 2$.  
Ce qui revient à demander l'existence d'une factorisation du type 
$$
F = \vartheta \fb 
\qquad 
\text{avec } \qquad 
\left\{
\begin{array}{l}
\vartheta \in E \text{ sans torsion} \\ [1.5mm]
\fb \text{ idéal de type fini de $\bfA$ tel que } \Gr(\fb) \geqslant 2
\end{array}
\right.
$$
Un pgcd fort, s'il existe, est unique à un inversible près.
\end{defn}

\medskip

\begin{prop} [Exponentiation et local-global]  
\label{PropProfondeur2}
Soit $\ua$ et $\ub$ deux suites et $e \in \bbN$.

\begin{enumerate}[\rm i)]
\item 
On a l'implication 
\quad $\Gr(a_1, \dots, a_m) \geqslant 2 \implies \Gr(a_1^e, \ldots, a_m^e) \geqslant  2$.

\item 
Pour tout idéal $\fa$ de type fini, on a 
\quad $\Gr\fa \geqslant  2 \implies \Gr\fa^e \geqslant  2$.

\item 
{\rm Local-global pour la divisibilité.} 
Soient $g,h \in \bfA$. On suppose $h$ régulier et $\Gr(\ub) \geqslant 2$.
On a l'implication
$$
h \mid g
\quad \text{dans chaque localisé $\bfA[1/b_i]$}
\quad \implies \quad 
h \mid g
$$
En particulier, pour un idéal $\fa$ vérifiant 
$\fa \subset \langle h \rangle$ dans chaque localisé, alors 
$\fa \subset \langle h \rangle$ dans $\bfA$.
\item 
{\rm Local-global pour la profondeur $\geqslant 2$.} 
Soit $\uc$ une suite. On suppose $\Gr(\ub) \geqslant 2$. On a l'implication
$$
\Gr(\uc) \geqslant 2 
\quad \text{sur chaque localisé $\bfA[1/b_i]$}
\quad \implies \quad 
\Gr(\uc) \geqslant 2 
$$
\end{enumerate}
\end{prop}

\begin {proof}

Dans la preuve, l'écriture $\uu \wedge \uv = 0$ pour deux suites de mêmes longueurs 
a la signification $\uu$ et~$\uv$ sont proportionnelles \idest{} $u_i v_j = u_jv_i$ pour tous $i,j$.

\smallskip
i)
Pour montrer 
$$
\Gr(a_1, \dots, a_m) \geqslant 2  \quad\Longrightarrow\quad
\Gr(a_1^e, \dots, a_m^e) \geqslant 2
$$
il suffit de montrer que pour $a, b \in \bfA$ et $\uc$ une suite de $m-1$ scalaires, 
on a :
$$
\Gr(a, \uc) \geqslant 2 
\quad \text{ et }\quad
\Gr(b, \uc) \geqslant 2 
\quad \implies \quad 
\Gr(ab, \uc) \geqslant 2 
$$
Le côté \og annulateur \fg{} \idest{} $\rmH^0(ab,\uc) = 0$, facile, est laissé au lecteur.
Occupons-nous à présent du $\rmH^1$ :
$$
\rmH^1(a,\uc) = 0  \hbox { et } \rmH^1(b,\uc) = 0   \quad\Longrightarrow\quad \rmH^1(ab,\uc) = 0 
$$
Soit $y \in \bfA^m$ proportionnel à $(ab,\uc)$. 
On veut montrer que $y$ est multiple de $(ab,\uc)$.
L'hypothèse
$$
\begin{bmatrix} y_1\\ y_2 \\\vdots\\ y_m \end {bmatrix} \wedge
\begin{bmatrix} ab\\ c_2 \\\vdots\\ c_m \end {bmatrix} \overset {(\star)}{=} 0
\quad \hbox{implique} \quad
\begin{bmatrix} y_1\\ by_2 \\\vdots\\ by_m \end {bmatrix} \wedge
\begin{bmatrix} a\\ c_2 \\\vdots\\ c_m \end {bmatrix} = 0 \, ;
\quad \hbox{et donc, grâce à $\rmH^1(a,\uc) = 0$, } \quad
\begin{bmatrix} y_1\\ by_2 \\\vdots\\ by_m \end {bmatrix} =
q \begin{bmatrix} a\\ c_2 \\\vdots\\ c_m \end {bmatrix} 
$$
pour un certain $q \in \bfA$.
En particulier $y_1 = qa$. 
En reprenant les $m-1$ dernières égalités de $(\star)$, 
et les $m-1$ dernières lignes de l'égalité précédente, on obtient :
$$
\begin{bmatrix} q\\ y_2 \\\vdots\\ y_m \end {bmatrix} \wedge
\begin{bmatrix} b\\ c_2 \\\vdots\\ c_m \end {bmatrix} = 0\,; 
\quad \hbox{comme $\rmH^1(b,\uc) = 0$, il existe $q'$ tel que} 
\ 
\begin{bmatrix} q\\ y_2 \\\vdots\\ y_m \end {bmatrix} 
= q' \begin{bmatrix} b\\ c_2 \\\vdots\\ c_m \end {bmatrix}
\quad \hbox{\idest{}} \quad
\begin{bmatrix} y_1\\ y_2 \\\vdots\\ y_m \end {bmatrix} 
= q' \begin{bmatrix} ab\\ c_2 \\\vdots\\ c_m \end {bmatrix}.
$$

\smallskip
ii) résulte de i).

\smallskip
iii) 
L'hypothèse se traduit par l'existence d'un exposant $e$ tel que
pour tout $i$, on ait $h \mid b_i^e g$ dans $\bfA$. 
D'après i), on a $\Gr(b_1^e, \ldots, b_m^e) \geqslant 2$.
Avec~\ref{Gr2StrongGcd}-i), on obtient $h \mid g$.

\smallskip
iv) Le côté annulateur \idest{} $\rmH^0(\uc) = 0$ est facile. 
Pour la nullité de $\rmH^1(\uc)$, considérons $\ux$ tel que $\ux \wedge \uc = 0$ sur $\bfA$.
On passe dans chaque 
localisé en $b_i$ et on obtient, d'après l'hypothèse, l'existence d'un exposant $e$ 
et d'une suite $\uq = (q_i)$ tels que
$$
\forall\, i,\ b_i^e \ux \ = \ q_i \uc 
\qquad \qquad \text{ ou encore } \qquad \qquad 
\forall\, j,\ \ub^e x_j \ = \ \uq c_j
$$
On a $\ub^e \wedge \uq = 0$ (en effet, pour tout $j$, on a 
$c_j (\ub^e \wedge \uq) = 0$ et $\Ann(\uc) = 0$).
Comme $\Gr(\ub^e) \geqslant 2$ (cf. i), il existe $\lambda \in \bfA$ tel que $\uq = \lambda 
\ub^e$. 
On a donc, pour tout $j$, $x_j = \lambda c_j$
(car $\ub^e x_j = \lambda \ub^e c_j$ et $\Ann(\ub^e) = 0$).
D'où $\ux = \lambda \uc$. 
\end {proof}

\label{NOTA01-ExpSuite}%
%
%

Dans ce contexte, tout a été fait pour que l'idéal $c\fa$ admette $cg$ comme pgcd fort
lorsque $c$ est régulier. 
Mais nous avons bien mieux comme le montre le point ii) suivant.

\begin{prop} [Multiplicativité de la propriété $\Gr \geqslant 2$ et du pgcd fort] 
\label{Gr2Multiplicativity} 
\leavevmode
\begin{enumerate}[\rm i)]
\item 
Etant données deux familles finies $\ua$ et $\ub$, en désignant par $\ua\,\ub$ la famille
de termes les $a_i b_j$:
$$
\big(
\Gr(\ua) \geqslant  2  \hbox { et } \Gr(\ub) \geqslant  2 
\big) 
\quad \implies \quad 
\Gr(\ua\,\ub) \geqslant  2
$$

\item 
Soient deux idéaux de type fini $\fa, \fb$ ayant respectivement des pgcds forts $g,h$.
Alors l'idéal produit $\fa\fb$ admet $gh$ comme pgcd fort.
\end{enumerate}
\end{prop}

\begin{proof} 
i) Posons $\uc = \ua\, \ub$. 
Comme $\Gr(\ua) \geqslant 2$, on a $\Gr(\uc) \geqslant 2$ sur chaque localisé $\bfA[1/b_i]$
(en effet, sur un tel localisé, l'idéal $\langle \uc \rangle$ contient la suite $\ua$ 
et on applique~\ref{Gr2ProprietesElementaires}).
Comme $\Gr(\ub) \geqslant 2$, le principe local-global~\ref{PropProfondeur2} s'applique 
et fournit $\Gr(\uc) \geqslant 2$.

ii) Par hypothèse, on a $\fa = g\langle\ua\rangle$ et $\fb = h\langle\ub\rangle$
avec $\Gr(\ua) \geqslant 2$ et $\Gr(\ub) \geqslant 2$. 
L'idéal produit vérifie donc $\fa \fb = gh \langle \ua \, \ub \rangle$.
On a $gh$ régulier et, d'après i), $\Gr(\ua \, \ub) \geqslant 2$.
Ainsi, par unicité du pgcd fort, l'idéal $\fa \fb$ admet $gh$ comme pgcd fort.
\end{proof}

\begin{prop}[Aspect \og local\fg{} du pgcd fort] \label{IdealPrincipalLocalGlobal}
Soit $\fa$ un idéal de type fini et un scalaire régulier $g \in \bfA$. 
On suppose que $g$ est un générateur de $\fa$ dans chaque localisé $\bfA[1/b_i]$ 
où $\ub$ est une suite de profondeur $\geqslant 2$.
Alors $g$ est un pgcd fort de $\fa$.
\end{prop}

\begin{proof}
On a $\fa \subset \langle g \rangle$ dans chaque localisé, donc dans $\bfA$
par principe local-global~\ref{PropProfondeur2}.
Donc il existe un idéal de type fini~$\fb$ tel que $\fa = g \fb$ dans $\bfA$.
Dans chaque localisé, on a $g \in \fa$ donc $1 \in \fb$.
Dans $\bfA$, cela se traduit par l'existence d'un $e$ tel que $\langle \ub^e \rangle \subset \fb$.
D'après le résultat d'exponentiation~\ref{PropProfondeur2}, on a $\Gr(\ub^e) \geqslant 2$, 
puis $\Gr(\fb) \geqslant 2$ 
(croissance de $\Gr$ pour l'inclusion,~cf.~\ref{Gr2ProprietesElementaires}).
\end{proof}

\begin{prop}[Unicité factorisation]
\label{uniciteFactorisation}
Soit $\ua = (a_1, \dots, a_n)$ une suite vérifiant $\Gr(\ua) \geqslant 1$ et 
$\ub = (b_1, \dots, b_m)$ une suite d'éléments de $\bfA$ vérifiant $\Gr(\ub) \geqslant 2$.
Même chose pour $\ua' = (a'_1, \dots, a'_n)$ et $\ub' = (b'_1, \dots, b'_m)$. 
On suppose que $\ua \, \ub  = \ua' \, \ub'$ \idest{} $a_i b_j = a'_i b'_j$.
Alors il existe un inversible $\varepsilon \in \bfA$ tel que 
$\ua = \varepsilon \ua'$ et $\ub = \varepsilon^{-1} \ub'$.
\end{prop}

\begin{proof}
Notons $\ub\wedge\ub'$ la famille des mineurs d'ordre 2 de la matrice
$2\times m$ de lignes $\ub$ et $\ub'$.  En remarquant que $a_j\ub =
a'_j\ub'$, on constate que l'on~a $a_ia'_j\,\ub\wedge\ub' = 0$ pour
tous $i, j$.  Comme $\Gr(\ua) \geqslant 1$ et $\Gr(\ua') \geqslant 1$,
on en tire $\ub \wedge \ub' =0$.  Maintenant, on utilise
successivement $\Gr(\ub) \geqslant 2$ et $\Gr(\ub') \geqslant 2$ : il
existe $\varepsilon, \varepsilon' \in \bfA$ tels que $\ub' =
\varepsilon\ub$ et $\ub = \varepsilon'\ub'$.  D'où $\ub = \varepsilon
\varepsilon' \ub$ puis $ \varepsilon \varepsilon' = 1$.
\end{proof}

\subsection{Le théorème de Wiebe}

Soit $\ua = (a_1, \dots, a_n)$ et $\ux = (x_1, \dots, x_n)$ deux suites de même longueur 
d'un anneau~$\bbA$ telles que $\langle \ua \rangle \subset \langle \ux \rangle$.
Alors, sans aucune hypothèse, on a la suite exacte suivante 
$$
\xymatrix @M=0.4pc @C=3.4pc{
\big(\bbA / \langle \ua \rangle\big)^n
\ar[r]^-{\left[
\begin{smallmatrix}
x_1 & \cdots & x_n \\
\end{smallmatrix}
\right]} 
&
\bbA/\langle \ua \rangle \ar[r]^-{\times 1} & 
\bbA/\langle \ux \rangle \ar[r] & 
0
}
$$
Considérons à présent une matrice 
$V \in \bbM_n(\bbA)$ certifiant l'inclusion $\langle \ua \rangle \subset \langle \ux \rangle$
au sens où 
$\begin{bmatrix}
a_1 & \cdots & a_n 
\end{bmatrix}
 =
\begin{bmatrix}
x_1 & \cdots & x_n 
\end{bmatrix} \, 
V$. 
On a alors une inclusion \og réciproque\fg{}, 
à savoir $(\det V) \cdot \langle \ux \rangle \subset \langle \ua \rangle$.
De manière imagée, le théorème de Wiebe assure l'exactitude
d'une suite analogue à celle ci-dessus pour cette inclusion réciproque, 
pourvu que $\ua$ et $\ux$ soient complètement sécantes.

\begin{theo}[Wiebe] \label{WiebeTheorem}
Soit $\ua = (a_1, \dots, a_n)$ et $\ux = (x_1, \dots, x_n)$ deux suites de même longueur 
d'un anneau~$\bbA$ telles que $\langle \ua \rangle \subset \langle \ux \rangle$ 
et $V$ une matrice telle que 
$\begin{bmatrix}
a_1 & \cdots & a_n 
\end{bmatrix}
 =
\begin{bmatrix}
x_1 & \cdots & x_n 
\end{bmatrix} \, 
V$.
Si $\ua$ est complètement sécante (alors $\ux$ l'est ipso facto), 
la suite ci-dessous est exacte
$$
\xymatrix @M=0.4pc @C=3pc{
0 \ar[r] & 
\bbA/\langle \ux \rangle \ar[r]^-{\times \det V} & 
\bbA/\langle \ua \rangle \ar[r]^-{\left[
\begin{smallmatrix}
x_1 \\ \vdots \\ x_n \\
\end{smallmatrix}
\right]} 
&
\big(\bbA / \langle \ua \rangle\big)^n
}
$$
Autrement dit, on a les égalités de transporteurs (idéaux de $\bbA$):
$$
(\ua : \det V) \ = \ \langle \ux \rangle
\qquad \qquad
\text{et}
\qquad \qquad
(\ua : \ux) \ = \ \langle \det V, \, \ua \rangle
$$
En posant $\bbB = \bbA / \langle \ua \rangle$, on peut encore
écrire en termes d'annulateurs (idéaux de l'anneau $\bbB$):
$$
\Ann(\det V \,; \bbB) \ = \ \ux \, \bbB
\qquad \qquad 
\text{et}
\qquad \qquad
\Ann(\ux \,; \bbB) \ = \ \det (V) . \,\bbB
$$
\end{theo}

\index{théorème!de Wiebe}%


\subsubsection*{Dans la littérature}

Il existe dans la littérature diverses preuves du théorème de Wiebe, à
commencer par celle de Wiebe lui-même \cite{Wiebe}.
Celle qui figure dans \cite[Appendice~E, corollary E.21, p.~356]{Kunz}
est limitée au cas noethérien et utilise l'arsenal de l'algèbre commutative noethérienne non
constructive (spectre premier, passage en local, évitement des idéaux
premiers); mais en l'analysant soigneusement, on peut l'adapter de manière à
obtenir une version non noethérienne avec \og utilisation du true-grade \fg{}
(suites régulières dans une extension polynomiale).
Une preuve générale figure dans Jouanolou \cite[section 3.8, Jacobien tordu]{J4} 
mais elle n'est pas destinée aux enfants car interviennent cohomologie locale 
(en fait de \Cech{}) et suites spectrales.

Signalons une preuve homologique relativement
élémentaire que l'on trouve dans l'article~\cite{SimonStrooker}.

Un point de vue totalement différent fait l'objet du court article
``An elementary proof of Wiebe's theorem'' \cite{CT-TC}.
Les principes d'une théorie de la profondeur, basés
sur la notion de ``True grade à la Northcott'', y sont abordés à l'aide de concepts
élémentaires, sans cohomologie ni suite exacte. Simplicité mais efficacité
puisque le théorème de Wiebe y est prouvé.

Dans l'annexe~\ref{WiebeProof}, 
nous proposons une preuve homologique nouvelle (à notre
connaissance), qui repose sur l'analyse d'un bicomplexe 
$\hbox{Koszul} \otimes \hbox{Koszul}$, 
ce qui nous semble aussi adapté que la technique de
\mbox{Jean-Pierre} Jouanolou utilisant la suite spectrale issue d'un bicomplexe de type 
$\hbox{Koszul} \otimes \hbox{\Cech}$. 
Il convient d'apporter quelques précisions sur ce point de vue et d'en
nuancer les propos. 
Le théorème de Wiebe peut se reformuler uniquement à l'aide d'annulateurs ;
ainsi utiliser Koszul semble naturel, puisque par définition $\rmH^0(\ux \, ; M)$ est le
sous-module des éléments de $M$ annulés par $\ux$ 
(tandis que le module de cohomologie de \Cech{} $\vH^0(\ux \, ; M)$ 
est le sous-module des éléments de $M$ annulés par une \textit{puissance} de $\ux$).
Cependant, il est important de signaler que
l'utilisation d'un complexe de type $\hbox {Kozzul} \otimes \hbox
{\Cech}$ conduit à un mécanisme portant des informations structurelles
beaucoup plus riches (la \textit{transgression}); nous y ferons allusion
dans la section~\ref{HCech0B}.

\medskip

En passant, une curiosité : on trouve chez Bourbaki 
\cite[\S4, exercice 6]{BourbakiACX}
un exercice consacré au théorème de Wiebe ; mais
l'anneau est local noethérien régulier, ce qui n'a pas grand chose à voir
avec ce que nous estimons être le véritable contexte du théorème de Wiebe. 
Et il faut noter l'usage du foncteur Ext: pourquoi faire simple quand on
peut faire compliqué~?

\subsubsection*{Indépendance de $\det V$ modulo $\langle\protect\ua\rangle$ vis-à-vis du choix de $V$}

\begin{prop} \label{IndependanceNabla}
Soient $V, V' \in \bbM_n(\bbA)$ telles que 
$$
\begin{bmatrix} a_1 &\cdots &a_n \end{bmatrix}
 = \begin{bmatrix} x_1 & \cdots & x_n \end{bmatrix} \, V
 = \begin{bmatrix} x_1 & \cdots & x_n  \end{bmatrix} \, V'
$$
Si la suite $\ux$ est $1$-sécante, alors:
$$
\det V -\det V' \ \in \ \langle \ua \rangle
$$
\end{prop}

\begin{proof}
Notons $V_1, \cdots, V_n$ les colonnes de $V$,  $V'_1, \cdots, V'_n$ celles de $V'$ et,
pour $0 \le i \le n$, définissons la matrice \og mixte\fg{} $V^{(i)} \in \bbM_n(\bbA)$
ayant pour colonnes
$$
V'_1, \cdots, V'_i,\ V_{i+1}, \cdots, V_n
$$
On a d'une part $\ua.V^{(i)} = \ux$ et d'autre part $V^{(0)} = V$,
$V^{(n)} = V'$. Il suffit donc de démontrer que $\det V^{(i)} - \det
V^{(i+1)} \in \langle\ua\rangle$ puisqu'en utilisant cette
appartenance de proche en proche, on obtiendra celle convoitée dans
l'énoncé.

On peut ainsi se ramener au cas où $V$ et $V'$ ne diffèrent que d'une colonne. 
Supposons par exemple que ce soit la première colonne et notons $h = V_1 - V'_1 \in \bbA^n$. 
Dans ce cas, on va montrer que $\det V - \det V' \in \langle a_2, \dots, a_n \rangle$ ce qui 
terminera la preuve.

Le développement de $\det V$ et $\det V'$ par rapport à la première colonne fournit
$\det V - \det V' = \sum h_j \Delta_j$  où $\Delta_j = \det (\varepsilon_j, V_2, \dots, V_n)$.
Par ailleurs, comme $h = V_1-V'_1$ et $\scp{V_1}{x} = \scp{V'_1}{x} = a_1$,
on a $\sum x_i h_i= 0$. Puisque la suite $\ux$ est $1$-sécante, il y a
une matrice alternée $N$  telle que $h_j = \sum_i x_i N_{ij}$.
En utilisant que $N$ est alternée, on obtient
$$
\begin{array}{rcl}
\det V - \det V' \ = \ 
\sum\limits_{j=1}^n h_j \Delta_j 
& = & 
\sum\limits_j \Big(\sum\limits_{i} x_i N_{ij} \Big) \Delta_j 
\ \, = \ \, \sum\limits_{i<j} x_i N_{ij} \Delta_j + \sum\limits_{i>j} x_i N_{ij} \Delta_j \\ [0.6cm]
& = & \sum\limits_{i<j} x_i N_{ij} \Delta_j +  \sum\limits_{i>j} x_i (-N_{ji}) \Delta_j  
\ \, = \ \, \sum\limits_{i<j} (x_i \Delta_j  - x_j  \Delta_i)N_{ij}
\end{array}
$$
En appliquant Cramer (cf la proposition~\ref{CramerSymetrie}) aux $n-1$ vecteurs $V_2, \cdots, V_n$
et aux deux vecteurs de la base canonique  $\varepsilon_i, \varepsilon_j$, 
on obtient $\Delta_i \varepsilon_j - \Delta_j \varepsilon_i \in \bbA V_2 + \cdots + \bbA V_n$.
Effectuons le produit scalaire avec $x$ en remarquant que, par définition,  
$a_i = \langle x \mid V_i \rangle$:
$$
\Delta_i \, x_j - \Delta_j \, x_i \ \in \ 
\langle a_2,\dots, a_n \rangle
$$
En reprenant l'expression de $\det V - \det V'$ ci-dessus, 
on obtient $\det V - \det V' \in \langle a_2, \dots, a_n \rangle$.
\end{proof}

\subsection{Contrôle de la profondeur par les composantes homogènes dominantes}

On commence par un exemple simple, celui de deux polynômes $F,G$, pour
montrer comment la propriété \og $(F,G)$ est une suite régulière\fg{} peut se déduire
de la même propriété sur les composantes homogènes dominantes. Voici l'énoncé précis.

\begin{prop}[Suite régulière de longueur 2]
\label{ControleProfCompHmgDom-nequal2}
Soit $\bfR[\uU]$ un anneau de polynômes sur un anneau quelconque
$\bfR$ et $F, G \in \bfR[\uU]$ avec $\deg F \leqslant f$ et $\deg G \leqslant g$.
On note $F'$ la composante homogène de degré $f$ de $F$ et $G'$ celle de degré $g$ de~$G$.
Si la suite $(F',G')$ est régulière alors il en est de même de la suite $(F,G)$.
\end{prop}

\begin{proof} 
Montrons tout d'abord que $F$ est régulier. Soit $H \in \bfR[\uU]$ tel que $FH = 0$.
Pour montrer que $H = 0$, il suffit de montrer que $\deg H \leqslant h \Rightarrow
\deg H \leqslant h-1$. En prenant la composante homogène de degré $f+h$ dans
l'égalité $FH = 0$, on obtient $F'H_h = 0$ et comme $F'$ est régulier,
on obtient $H_h = 0$, ce qu'il fallait démontrer.

\smallskip

Montrons qu'une égalité $UF = VG$ entraîne l'existence d'un polynôme $Q$
tel que $\begin{bmatrix} U\\ V\\\end{bmatrix} = Q\, \begin{bmatrix} G\\ F\\\end{bmatrix}$.
On procède par récurrence sur $d$ où l'entier $d$ vérifie :
$$
\deg U \leqslant d+g
\qquad \text{et} \qquad
\deg V \leqslant d+f
$$
La considération de la composante homogène de degré $d+f+g$ dans l'égalité $UF = VG$
fournit
$$
U_{d+g} F' \,=\, V_{d+f} G'
$$
Puisque la suite $(F',G')$ est régulière, il existe un polynôme $Q_d$
tel que $V_{d+f} = Q_dF'$, donc $U_{d+g} = Q_dG'$ ;
on peut supposer que le polynôme $Q_d$ est homogène de degré $d$,
quitte à considérer sa composante homogène de degré $d$.
L'idée est de remplacer $U$ par $U - Q_dG$ et $V$ par $V - Q_dF$.
On a 
$$
(U-Q_dG)F \,=\, (V-Q_dF)G
$$
De plus, $(U-Q_dG)_{d+g} = U_{d+g} - Q_dG' = 0$ et $(V-Q_dF)_{d+f} = V_{d+f} - Q_dF' = 0$
donc $\deg (U-Q_dG) < d+g$ et $\deg (V-Q_dF) < d+f$. La récurrence s'applique et
fournit un polynôme $Q$ tel que
$$
U-Q_dG = QG, \qquad V-Q_dF = QF   \qquad \hbox {d'où} \qquad
U = (Q_d + Q)G, \qquad V = (Q_d + Q)F
$$
ce qui termine la preuve.
\end {proof}

Le résultat suivant, qui expose comment la profondeur d'une famille finie
de polynômes peut être contrôlée par celle de ses composantes homogènes
dominantes, sera utilisé à de nombreuses reprises.
Dans le dernier point, nous nous permettons d'utiliser le mot ``suite''
au lieu de ``famille'' car les suites qui interviennent sont
commutativement régulières.

\begin{theo}[Contrôle de la profondeur par les composantes homogènes dominantes]
\label{ControleProfondeur}
Soit $\bfR[\uU]$ un anneau de polynômes sur un anneau quelconque
$\bfR$ et une famille finie de polynômes $\uF = (F_i)$ de~$\bfR[\uU]$. 
On suppose que $F_i$ est de degré $\leqslant f_i$. On note $F'_i$ sa
composante homogène de degré~$f_i$ et $\uF'$ la famille~$(F'_i)$.

\medskip

\begin{enumerate}[\rm i)]
\item
Pour un $p$ donné, 
on a l'implication $\rmH^p(\uF' \,; \bfR[\uU]) = 0 \ \Rightarrow \ \rmH^p(\uF \,; \bfR[\uU]) = 0$.

\item
Pour un $m$ donné, 
on a l'implication $\rmH_m(\uF' \,; \bfR[\uU]) = 0 \ \Rightarrow \ \rmH_m(\uF \,; \bfR[\uU]) = 0$.
En particulier pour $m=1$, si $\uF'$ est 1-sécante, il en est de même de $\uF$.

\item 
Si $\Gr(\uF' \,; \bfR[\uU]) \geqslant q$, alors $\Gr(\uF \,; \bfR[\uU]) \geqslant q$.

\item 
Si la suite $\uF'$ est régulière et les $f_i$ strictement positifs, alors
la suite $\uF$ est commutativement régulière.
\end{enumerate}
\end {theo}

\index{théorème!de contrôle de la profondeur par!1@les composantes homogènes dominantes}%

\begin {proof}[Preuve]
Via l'auto-dualité du complexe de Koszul (proposition~\ref{KoszulAutoDualite}),
les points i) et ii) sont équivalents.
Par définition de la profondeur, le point iii) est une conséquence du point~i).  
Le point iv) résulte du point~ii) en vertu des implications suivantes :
$$
\begin{array}{rcl}
\text{$\uF'$ est régulière avec $f_i > 0$} 
& \overset{\scriptstyle \ref{PermutabiliteSuiteReguliereHomogene}}{\implies} & 
\text{$\uF'$ est commutativement régulière} \\
& \overset{\ref{ReguliereSuiteExtraite1Secante}}{\implies} & 
\text{les suites extraites de $\uF'$ sont 1-sécantes} \\
& \overset{\rm ii)}{\implies} & 
\text{les suites extraites de $\uF$ sont 1-sécantes} \\
& \overset{\ref{ReguliereSuiteExtraite1Secante}}{\implies} & 
\text{$\uF$ est commutativement régulière} \\
\end{array}
$$
Il reste donc à montrer le point i). Afin de simplifier les notations, 
nous avons choisi d'en fournir la preuve uniquement dans les cas $p = 0,1,2$.  
Celle-ci repose essentiellement sur la remarque (banale)
suivante concernant la composante homogène d'un produit de deux
polynômes :
$$
(FG)_{f+g} \, =\,  F_f\,G_g  \qquad\quad
\hbox {lorsque $\deg F \leqslant f$ et $\deg G \leqslant g$} 
$$
où l'on note $F_d$ la composante homogène de degré $d$ de $F$.
\end{proof}

\begin{proof}[\bfseries Preuve du point {\rm i)} dans le cas $p=0$]

\leavevmode

\noindent
Il s'agit de montrer que $\Ann(\uF') = 0 \Rightarrow \Ann(\uF) = 0$.
Soit $G$ tel que $GF_i = 0$ pour tout $i$. 
Pour prouver que $G = 0$,
il suffit de prouver que $\deg G \leqslant d \Rightarrow \deg G \leqslant d-1$.
Dans l'égalité $GF_i = 0$, prenons la composante homogène de degré $d+f_i$.
On obtient $G_d F'_i = 0$ pour tout $i$. Comme $\Ann(\uF') = 0$, on obtient $G_d = 0$ \idest{} 
$\deg G\leqslant d-1$.
\end {proof}

\begin {proof}[\bfseries Preuve du point {\rm i)} dans le cas $p=1$]\leavevmode

\noindent
Soit $d \in \mathbb Z$.
On considère l'hypothèse de récurrence $\calH_d$ suivante :

\begin{quote}
\textit{
Pour toute famille $\uG = (G_i)$ de $\bfR[\uU]$ vérifiant
$$
\forall\, i,\  \deg G_i \leqslant d + f_i 
\qquad \text{et} \qquad 
\forall\, i,j, \ G_jF_i = G_iF_j
$$
il existe $Q \in \bfR[\uU]$ tel que $G_i = QF_i$ pour tout $i$.
}
\end{quote}

\noindent
$\rhd$ On initialise la récurrence avec $d = \min (-f_i)-1$ de sorte que $d + f_i < 0$ pour tout $i$.
Ainsi chaque polynôme~$G_i$ est nul et on peut prendre $Q= 0$.

\noindent
$\rhd$ Soit $d$ un entier et $\uG$ une famille vérifiant $G_j F_i = G_i F_j$ 
avec $\deg G_i \leqslant d + f_i$.
Puisque $\deg G_j \leqslant d + f_j$ et $\deg F_i \leqslant f_i$, la composante
homogène de degré $d + f_j + f_i$ de $G_j F_i$ est:
$$
(G_jF_i)_{d + f_j + f_i} 
\ =\ 
(G_j)_{d+f_j} (F_i)_{f_i}
\ \overset{\rm def}{=} \ 
G'_jF'_i
\qquad \hbox {où $G'_j = (G_j)_{d + f_j}$}
$$
En prenant la composante homogène de degré $d + f_j + f_i$ dans 
$G_jF_i = G_i F_j$, on obtient donc :
$$
G'_j\,F'_i \, =\,  G'_i\,F'_j
$$
Puisque $\rmH^1(\uF' \, ; \bfR[\uU]) = 0$, il existe un polynôme $T$ tel que
$G'_i = T\,F'_i$.
Le polynôme $G'_i$ étant homogène de degré $d+f_i$ et $F'_i$ homogène de degré $f_i$, 
on peut supposer
que $T$ est homogène de degré $d$ 
(quitte à le remplacer par sa composante homogène de degré $d$ comme on le voit
en considérant la composante homogène de degré~$d+f_i$ dans l'égalité $G'_i = T\,F'_i$).

On introduit alors $R_i = G_i - TF_i$; il est clair que
$\deg R_i \leqslant d+f_i$ et l'inégalité est stricte puisque:
$$
(G_i - TF_i)_{d+f_i}  \ =\  G'_i - TF'_i \ =\  0
$$
On a $\deg R_i \leqslant d-1 + f_i$ et $R_j F_i = R_i F_j$. 
D'après $\calH_{d-1}$, on en déduit qu'il existe un polynôme~$\widetilde Q$ tel que $R_i = \widetilde QF_i$.
On obtient donc $G_i = (T + \widetilde Q) F_i$ c'est-à-dire $G_i = Q F_i$ en posant $Q = T + \widetilde Q$.
\end{proof}

\begin {proof}[\bfseries Preuve du point {\rm i)} dans le cas $p=2$]\leavevmode

\noindent
Soit $d \in \mathbb Z$.
On considère l'hypothèse de récurrence $\calH_d$ suivante :

\begin{quote}
\textit{
Pour toute famille $\uG = (G_{ij})_{i<j}$ de $\bfR[\uU]$ vérifiant
$$
\forall\, i < j,\  \deg G_{ij} \leqslant d + f_i + f_j
\qquad \text{et} \qquad 
\forall\, i < j < k, \ 
G_{jk}\,F_i - G_{ik}\,F_j + G_{ij}\,F_k = 0
$$
il existe une famille $\uQ = (Q_i) \in \bfR[\uU]$ telle que $G_{ij} = Q_jF_i - Q_i F_j$ pour $i < j$.
}
\end{quote}

\noindent
$\rhd$ On initialise la récurrence avec $d = \min \big(-(f_i+f_j)\big)-1$ de sorte que 
$d + f_i + f_j < 0$ pour tous $i, j$.
Ainsi chaque polynôme~$G_{ij}$ est nul et on peut prendre $Q_i = 0$.

\noindent
$\rhd$ Soit $d$ un entier et $\uG = (G_{ij})$ 
une famille vérifiant $G_{jk}\,F_i - G_{ik}\,F_j + G_{ij}\,F_k = 0$
avec $\deg G_{ij} \leqslant d + f_i + f_j$.
Puisque $\deg G_{ij} \leqslant d + f_i + f_j$ et $\deg F_k \leqslant f_k$, la composante
homogène de degré $d + f_i + f_j + f_k$ de $G_{ij} F_k$ est :
$$
(G_{ij}F_k)_{d + f_i+f_j+f_k} 
\ =\  
G'_{ij}\,(F_k)_{f_k}
\ \overset{\rm def}{=} \ 
G'_{ij}\,F'_k
\qquad
\hbox {où $G'_{ij} = (G_{ij})_{d+f_i+f_j}$}
$$
En prenant la composante homogène de degré $d + f_i + f_j + f_k$ dans 
$G_{jk}\,F_i - G_{ik}\,F_j + G_{ij}\,F_k = 0$,
on obtient donc :
$$
G'_{jk}\,F'_i - G'_{ik}\,F'_j + G'_{ij}\,F'_k \ =\  0
$$
Puisque $\rmH^2(\uF' \, ; \bfR[\uU]) = 0$, il existe une famille $(T_i)$ vérifiant
$G'_{ij} = T_j\,F'_i - T_i\,F'_j$.
Le polynôme $G'_{ij}$ étant homogène de degré $d+f_i+f_j$ et $F'_j$ homogène de degré $f_j$, 
on peut supposer que $T_i$ est homogène de degré $d+f_i$, quitte à le remplacer par
sa composante homogène de degré~$d+f_i$.

On introduit alors $R_{ij} = G_{ij} - (T_jF_i- T_i F_j)$ ; il est clair que
$\deg R_{ij} \leqslant d+f_i + f_j$ et l'inégalité est stricte puisque :
$$
\big(G_{ij} - (T_jF_i - T_i F_j)\big)_{d+f_i+f_j} 
\ = \ 
G'_{ij} - (T_jF'_i- T_i F'_j)
\ = \ 
0
$$
On a $\deg R_{ij} \leqslant d-1 + f_i + f_j$ et $R_{jk} F_i - R_{ik} F_j + R_{ij} F_k = 0$. 
D'après $\calH_{d-1}$, on en déduit qu'il existe une famille~$(\widetilde Q_{i})$ telle que 
$R_{ij} = \widetilde Q_j F_i- \widetilde Q_i F_j$.
En posant $Q_i = T_i + \widetilde Q_i$, on obtient $G_{ij} = Q_j F_i-Q_i F_j$.
\end{proof}

\subsection{Le théorème d'Hilbert-Burch}
\label{sectionHB}

Nous allons utiliser un cas particulier du résultat général suivant dont on trouvera une preuve en
\cite [II.5.22, point 2.]{LombardiQuitte} ou bien, dans la seconde édition, en XV.8.7.

\begin {theo} [de McCoy]
Soit $u : \bfA^m \to \bfA^n$ une application linéaire. Alors elle est injective
si et seulement si l'idéal déterminantiel $\calD_n(u)$ engendré par les mineurs
d'ordre $n$ de $u$ est un idéal fidèle \idest{} $\Ann\,\calD_n(u) = 0$.
\end {theo}

\index{théorème!de McCoy}%

\bigskip

Le contexte est celui d'une application linéaire $u : \bfA^{n-1} \to \bfA^n$, de matrice $U$,
dont on note $\Delta = (\Delta_1, \dots, \Delta_n) : \bfA^n \to \bfA$ la forme linéaire
des mineurs signés, où les $\Delta_i$ sont définis par:
$$
\begin {vmatrix}
u_{1,1} &\cdots  &u_{1,n-1} & X_1\cr
u_{2,1} &\cdots  &u_{2,n-1} & X_2\cr
\vdots &\cdots  &\vdots \cr
u_{n,1} &\cdots  &u_{n,n-1} & X_n\cr
\end {vmatrix}
= \Delta_1X_1 + \cdots + \Delta_nX_n
$$
En désignant par $U_1, \cdots, U_{n-1} \in \bfA^n$ les $n-1$ colonnes de $U$, on peut donc voir
$\Delta$ de la manière suivante.
$$
\Delta(y) = \sum_i y_i\Delta_i = \det(U_1, \dots, U_{n-1},y)    \qquad  y \in \bfA^n
$$
On remarquera que l'idéal $\langle \Delta_1, \dots, \Delta_n\rangle$ est à la fois
l'idéal déterminantiel $\calD_1(\Delta)$ d'ordre 1 de $\Delta$ et l'idéal déterminantiel
$\calD_{n-1}(u)$ d'ordre $n-1$ de $u$.
On a $\Delta \circ u = 0$ d'où un petit complexe:
$$
0 \to \xymatrix {\bfA^{n-1} \ar[r]^u & \bfA^n \ar[r]^{\Delta} & \bfA}
\leqno (\star)
$$
On désigne par $(e_1, \dots, e_n)$ la base canonique de $\bfA^n$.
Notons que l'on a toujours $\Delta_k \Ker\Delta \subset \Im u$. En effet,
pour $y,y' \in \bfA^n$, on a $\Delta(y')y - \Delta(y)y' \in \Im u$ d'après
Cramer (cf. la proposition~\ref{CramerSymetrie}).
En prenant $y \in \Ker\Delta$ et $y' = e_k$, on obtient $\Delta_k y \in \Im u$.

\label{NOTA01-calD}%

\medskip
\noindent
Le théorème d'Hilbert-Burch étudie entre autres l'exactitude du
complexe~$(\star)$.  On peut le déduire de la théorie générale des
résolutions libres finies mais nous avons tenu ici à en fournir une
preuve directe indépendante.  Certains ingrédients qui interviennent
dans les preuves sont classiques en théorie des résolutions libres,
d'autres sont spécifiques au cas particulier étudié.  Bien
qu'élémentaire, ce résultat permet de dégager le lien entre
résolutions de longueur 2 et profondeur $\geqslant 2$.

\begin {theo} [d'Hilbert-Burch] \label{HilbertBurch} \leavevmode

\begin{enumerate}[\rm i)]
\item 
Si le complexe $(\star)$ est exact, alors $\Gr(\Delta) \geqslant 2$.

\item 
Réciproquement, si $\Gr(\Delta) \geqslant 2$, alors le complexe $(\star)$ est exact.

\item  
Sous le couvert de {\rm ii)} (ou {\rm i)} qui lui est équivalent), le sous
$\bfA$-module de $(\bfA^n)^\star$ constitué des formes linéaires $\mu$
nulles sur $\Im u$ est libre de rang 1 de base $\Delta$.

\item 
Soit $\mu : \bfA^n \to \bfA$ une forme linéaire telle que la suite ci-dessous
soit un complexe exact:
$$
\xymatrix {0 \ar[r]  & 
\bfA^{n-1} \ar[r]^u & \bfA^n \ar[r]^{\mu} & \bfA}
$$
Alors $\Gr(\Delta) \geqslant 2$ et $\mu$ est multiple de $\Delta$. 
\end{enumerate}
\end {theo}

\index{théorème!d'Hilbert-Burch}%

\begin {lem} \leavevmode

\noindent
Soit une suite exacte dans laquelle $E,F,G$ sont des $\bfB$-modules quelconques
sur un anneau commutatif~$\bfB$:
$$
0 \to \xymatrix {E \ar[r]^u &F \ar[r]^v & G}
$$
Alors pour tout élément \textrm{$G$-régulier} $b \in \bfB$, l'application
linéaire $\overline u : E/bE \to F/bF$ reste injective.
\end {lem}

\begin {proof} \leavevmode
Il faut prouver que pour $x \in E$ vérifiant $u(x) \in bF$, on a $x \in bE$. On
écrit $u(x) = by$ avec $y \in F$. Et on applique $v$, ce qui donne $bv(y) = 0$
donc $v(y) = 0$ car $b$ est $G$-régulier. Comme la suite est
exacte, il existe $x' \in E$ tel que $u(x') = y$. Il vient $u(bx') = by = u(x)$
donc $bx' = x$ puisque $u$ est injective.
\end {proof}

\begin {proof} [\bf Preuve du théorème d'Hilbert-Burch]
\leavevmode

\noindent
i) D'après le théorème de McCoy, on a $\Gr(\Delta) \geqslant 1$ puisque $u$ est injectif.
Soient $a_1, \cdots, a_n$ des scalaires vérifiant $\Delta_j a_i
= \Delta_i a_j$; on doit montrer l'existence de $q \in \bfA$ vérifiant
$a_i = q\Delta_i$ pour tout $i$. Ou encore, en définissant les
polynômes $f = \sum_i \Delta_i X_i$ et $g = \sum_i a_i X_i$, que $g$
est multiple de $f$.

\noindent
Le complexe obtenu par extension des scalaires de $\bfA$ à $\bfA[\uX]$ 
$$
0 \to \xymatrix {\bfA[\uX]^{n-1} \ar[r]^u & \bfA[\uX]^n \ar[r]^{\Delta} & \bfA[\uX]}
\qquad \hbox {reste exact.}
$$
Nous sommes dans les conditions d'appliquer le lemme en
prenant $\bfB = \bfA[\uX]$ et $b = f$, qui est bien un élément
régulier de $\bfA[\uX]$ (lemme de McCoy~\ref{McCoyPolyLemma}).
On obtient que
$\overline u : \bfB^{n-1}/f\bfB^{n-1} \to \bfB^n/f\bfB^n$ est
injectif. Et donc l'idéal déterminantiel $\langle \overline
{\Delta_1}, \dots, \overline {\Delta_n}\rangle$, idéal de
$\bfB/f\bfB$, est fidèle.

\noindent
Mais on a $\Delta_i g = a_i f$, a fortiori $\langle \overline
{\Delta_1}, \dots, \overline {\Delta_n}\rangle \overline g = 0$ donc,
par fidélité de $\langle \overline {\Delta_1}, \dots, \overline
{\Delta_n}\rangle$, on a $\overline g = 0$. On vient de montrer que
$g$ est multiple de $f$.

\medskip\noindent
ii) L'utilisation de $\Gr(\Delta) \geqslant 1$ fournit déjà, via le théorème de McCoy,
le fait que $u$ est injective.

\noindent
Soit $y \in \Ker\Delta$; on veut trouver $x \in \bfA^{n-1}$ tel que $u(x) = y$. Dans le préambule
avant le théorème, on a vu que pour $j = 1, \dots, n$, on a $\Delta_j y \in \Im u$ d'où
un $x^{(j)} \in \bfA^{n-1}$ tel que $\Delta_j y = u(x^{(j)})$.  On en déduit:
$$
u(\Delta_i x^{(j)}) = \Delta_i \Delta_j y = \Delta_j \Delta_i y = u(\Delta_j x^{(i)})
\qquad \hbox {et comme $u$ est injective} \qquad \Delta_i x^{(j)} = \Delta_j x^{(i)}
$$
On utilise maintenant l'hypothèse $\Gr(\Delta) \geqslant 2$, que l'on applique aux égalités
obtenues en prenant la composante d'indice 1 de l'égalité vectorielle de droite.
Ceci fournit un scalaire $x_1$ tel que $x^{(i)}_1 = x_1\Delta_i$. Idem avec les
composantes d'indices $2, 3, \cdots, n$ si bien que l'on obtient un vecteur
$x \in \bfA^n$ vérifiant $x^{(i)}_k = x_k\Delta_i$ pour tous $i,k$. Ceci s'écrit
vectoriellement $x^{(i)} = \Delta_i x$. Il vient alors
$$
\Delta_i y = u(x^{(i)}) = \Delta_i u(x) \quad \hbox {pour tout $i$}
$$
donc $y = u(x)$. On vient d'obtenir $y \in \Im u$, ce qui termine la preuve
de ce point.

\medskip\noindent
iii) D'après Cramer, on a $\Delta_j e_i - \Delta_i e_j \in \Im u$ (et
ceci sans aucune hypothèse).  Comme $\mu$ est nulle sur $\Im u$, on a
$\Delta_j\mu(e_i) = \Delta_i\mu(e_j)$. Puisque $\Gr(\Delta) \geqslant 2$, il
y a $q \in \bfA$ tel que $\mu(e_i) = q\Delta_i$; comme $\Delta_i
= \Delta(e_i)$, on obtient $\mu = q\Delta$.

\medskip\noindent
iv) Il suffit de montrer que $\Ker\mu =\Ker\Delta$, ce qui fera que la suite $(\star)$
est exacte et on pourra alors appliquer les points précédents.

\noindent
Posons $a_i = \mu(e_i)$. On a $a_i\Delta_j = a_j\Delta_i$ comme on le
voit en utilisant que $\Delta_j e_i - \Delta_i e_j \in \Im u
= \Ker \mu$.  On a toujours $\Im u \subset \Ker\Delta$ donc
$\Ker\mu \subset \Ker\Delta$.  Pour l'inclusion opposée, soit
$x \in \Ker\Delta$  \idest{} 
$\sum_i x_i \Delta_i = 0$.  En multipliant par
$a_1$ et en utilisant $a_1\Delta_i = a_i\Delta_1$, il vient
$\Delta_1 \sum_i x_i a_i = 0$. Mais ceci vaut pour tout $\Delta_j$ à
la place de $\Delta_1$. Comme l'idéal déterminantiel $\calD_{n-1}(u)
= \langle \Delta_1, \dots, \Delta_n\rangle$ est fidèle (car $u$ est
injectif), il vient $\sum_i x_i a_i = 0$ \idest{}  $x \in \ker\mu$,
ce qu'il fallait montrer.
\end {proof}

\subsection{Fonctions symétriques élémentaires et profondeur de l'idéal de Veronese}

Le résultat que nous allons établir concernant la profondeur de l'idéal de Veronese
sera utilisé à la fin du chapitre~\ref{ChapObjetsSylvester} pour y montrer
le résultat suivant (cf le théorème~\ref{ProfondeurWh}):
$$
\Gr\big(\det W^\sigma_{h,d}(\uP), \sigma \in \fS_n \big) \ge h
$$
Le résultat analogue pour les matrices $B^\sigma_{k,d}(\uP)$ fera l'objet du théorème~\ref{ProfondeurBk}.

\bigskip

Soit $\bfR [\uY] := \bfR[Y_1,\ldots,Y_n]$ un anneau de polynômes sur
un anneau quelconque $\bfR$.  Pour $1 \leqslant r \leqslant n$,
l'idéal de Veronese $\scrV_r$ de degré $r$ est l'idéal monomial de
$\bfR[Y_1,\ldots,Y_n]$ engendré par les monômes sans facteur carré de
degré~$r$ :
$$
\scrV_r
\ =\  
\Bigl \langle 
\, \prod_{i \in I} Y_i, \  \#I = r \, 
\Bigr \rangle
$$
On peut prolonger la définition pour tout $r \in \bbN$ de sorte que :
$$
\scrV_0 = \bfR[\uY] 
\ \supset \ 
\scrV_1 = \langle \uY \rangle 
\ \supset \quad \cdots \quad \supset \ 
\scrV_n = \langle Y_1 \cdots Y_n \rangle 
\ \supset \ 
\scrV_{n+1} = 0
$$
Notons $\sigma_1, \dots, \sigma_n$ les $n$ fonctions 
symétriques élémentaires en les $n$ variables $Y_1, \dots, Y_n$:
$$
\sigma_k = \sum\limits_{\#I=k} \prod\limits_{i \in I} Y_i
$$
On a donc $\sigma_k \in \scrV_r$ pour tout $k \geqslant r$.
On peut prolonger la définition des fonctions symétriques élémentaires pour tout $k \in \bbN$ :
$$
\sigma_0 = 1 \quad \text{et}\quad  \forall\, k > n, \ \sigma_k = 0 
\qquad \hbox {en cohérence avec} \qquad
\prod_{k=1}^n (1 + tY_k) = \sum_{k\geqslant0} \sigma_k\,t^k
$$

\label{NOTA01-VeroneseIdeal}%
\index{idéal!de Veronese}%

Voici l'énoncé que nous voulons démontrer :

\begin{theo} [Profondeur de l'idéal de Veronese $\scrV_r$] 
\label{ProfondeurVeronese}
L'idéal $\scrV_r$ de $\bfR[Y_1, \ldots, Y_n]$ contient une suite régulière homogène de longueur $n-(r-1)$. 
De manière précise, en notant $\sigma_1, \ldots, \sigma_n$
les $n$ fonctions symétriques élémentaires de $Y_1, \ldots, Y_n$ 
et $f(T)$ le polynôme :
$$
f(T) = T^n - \sigma_1 T^{n-1} + \cdots + (-1)^{r-1} \sigma_{r-1} T^{n-(r-1)}
$$
alors la suite $\big(f(Y_r), f(Y_{r+1}), \ldots, f(Y_n)\big)$ 
est une suite régulière homogène de longueur $n-(r-1)$ contenue dans l'idéal $\scrV_r$.
En particulier
$$
\Gr(\scrV_r) \, \geqslant \, n-(r-1)
$$
\end{theo}

Pour la preuve, nous aurons besoin de quelques propriétés des $\sigma_k$.
En voici quelques-unes pour commencer.
Les deux premières sont très classiques, 
cf.~\cite[section \S6 (Fonctions symétriques), th~1]{BourbakiAlgIV}.
Les deux autres le sont moins (disons, avec un anneau $\bfR$ général à la base).

\begin{enumerate}
\item 
Les éléments $\sigma_1, \ldots, \sigma_n$ sont algébriquement indépendants
sur~$\bfR$. 

\item 
Le module 
$\bfR[\uY]$ est libre sur $\bfR[\sigma_1, \ldots, \sigma_n]$ de rang $n!$ (de base les
$Y_1^{\alpha_1} \cdots Y_n^{\alpha_n}$ avec $0 \leqslant \alpha_i < i$).

\item 
La suite $(\sigma_1, \ldots, \sigma_n)$ est une suite régulière de $\bfR[\uY]$.

{\footnotesize
D'après 1, cette suite est régulière dans $\bfR'= \bfR[\sigma_1, \dots, \sigma_n]$.
Or, d'après 2, $\bfR[\uY]$ est libre sur $\bfR'$, donc 
d'après~\ref{RegTrivialFacts}-ii, la suite est régulière sur $\bfR[\uY]$.
}

\item 
En notant $\sigma'_i$ les fonctions symétriques élémentaires en les variables $Y_1, 
\dots, Y_{r-1}$, on a
$$
\bfR[\sigma_1, \ldots, \sigma_{r-1},Y_r, \dots, Y_n] 
= \bfR[\sigma'_1, \ldots, \sigma'_{r-1},Y_r, \dots, Y_n]
$$

{\footnotesize 
Notons $(\sigma''_j)$ les fonctions symétriques élémentaires en les variables $Y_r, \dots, Y_n$.
En considérant les coefficients de l'égalité :
$$
\prod_{k=1}^n (1 + tY_k) \ =\ 
\prod_{i=1}^{r-1}(1 + tY_i)\ \prod_{j=r}^n (1 + tY_j)
$$
on en déduit les égalités $\sigma_k = \sum_{i+j=k} \sigma'_i \sigma''_j$.
Ces égalités font apparaître une matrice carrée de taille $r$
à coefficients dans $\bfR[Y_r, \dots, Y_n]$, triangulaire inférieure à diagonale unité :
$$
\begin{bmatrix}
1 & \\ 
\sigma''_1 & 1 & \\ 
\sigma''_2 & \sigma''_1 & 1 & \\ 
\vdots & & & \ddots & \\
\sigma''_{r-2} & \sigma''_{r-3} & \cdots & \sigma''_1 & 1 &  \\
\sigma''_{r-1} & \sigma''_{r-2} & \cdots & \sigma''_2 & \sigma''_1 & 1 \\
\end{bmatrix} 
\begin{bmatrix}
\sigma'_0 \\ 
\sigma'_1 \\ 
\sigma'_2 \\ 
\vdots \\
\sigma'_{r-2} \\ 
\sigma'_{r-1} \\ 
\end{bmatrix}
\ = \ 
\begin{bmatrix}
\sigma_0 \\ 
\sigma_1 \\ 
\sigma_2 \\ 
\vdots \\
\sigma_{r-2} \\ 
\sigma_{r-1} \\ 
\end{bmatrix}
$$
de sorte que, pour tout $k \in \llbracket 1, r-1 \rrbracket$, on a 
$$
\sigma_k \in \bfR[\sigma'_1,\dots,\sigma'_{r-1},Y_r, \dots, Y_n]
\qquad \text{ et } \qquad 
\sigma'_k \in \bfR[\sigma_1,\dots,\sigma_{r-1},Y_r, \dots, Y_n]
$$
}
\end{enumerate}

\begin{lem} \label{LemmeFonctionsSymElem}
Soit $r \in \llbracket 1, n+1 \rrbracket$.
\leavevmode

\begin{enumerate}[\rm i)]

\item 
$(\sigma_1, \ldots, \sigma_{r-1}, Y_r, \ldots, Y_n)$ est une suite 
régulière de $\bfR[\uY]$.

\item 
Les éléments $\sigma_1, \ldots, \sigma_{r-1}, Y_r, \ldots, Y_n$ sont algébriquement 
indépendants sur $\bfR$. 

\item 
L'anneau $\bfR[\uY]$ est
un $\bfR[\sigma_1, \ldots, \sigma_{r-1}, Y_r, \ldots, Y_n]$-module libre.
\end{enumerate}
\end{lem}

\begin{proof}
i) 
Les polynômes de la suite sont homogènes de degré $>0$.
D'après~\ref{PermutabiliteSuiteReguliereHomogene},
il suffit d'examiner la suite permutée $(Y_r, \dots, Y_n, \sigma_1, \dots, \sigma_{r-1})$.
Il est clair que $(Y_r, \dots, Y_n)$ est régulière.
Modulo $\langle Y_r, \dots, Y_n \rangle$, la suite $\sigma_1,\dots, \sigma_{r-1}$ 
l'est également car il s'agit de la suite 
des $r-1$ fonctions symétriques élémentaires de l'anneau $\bfR[Y_1, \dots, Y_{r-1}]$
et on utilise le point 3 précédent.

ii) 
D'après le lemme~\ref{RegulariteIndependanceAlgebrique} qui suit et le point i),
on en déduit que les éléments 
$\sigma_1, \ldots, \sigma_{r-1}, Y_r, \ldots, Y_n$ sont algébriquement indépendants sur $\bfR$. 

iii) 
Posons $\bfS = \bfR[Y_r, \dots, Y_n]$.
D'après le point 4, on a 
$\bfS[\sigma_1, \ldots, \sigma_{r-1}] = \bfS[\sigma'_1, \ldots, \sigma'_{r-1}]$.
Et d'après le point 2, $\bfS[Y_1, \dots, Y_{r-1}]$ est un 
$\bfS[\sigma'_1, \ldots, \sigma'_{r-1}]$-module libre.
\end{proof}


\begin{lem}[Régularité et indépendance algébrique]
\label{RegulariteIndependanceAlgebrique}
Soit $\ua = (a_1, \dots, a_n)$ une suite régulière d'un anneau $\bfS$ et $\bfR \subset \bfS$ 
un sous-anneau tel que $\bfR \cap \langle \ua \rangle = \{ 0 \}$.
Alors $a_1, \dots, a_n$ sont algébriquement indépendants sur $\bfR$.
\end{lem}

\begin{proof}
On procède par récurrence sur $n$. Pour $n=1$, la suite est réduite à un élément $a$.
Considérons une relation $\lambda_m a^m + \cdots + \lambda_1 a + \lambda_0 = 0$
avec $\lambda_i \in \bfR$.
On a $\lambda_0 \in \bfR \cap \langle a \rangle = \{ 0\}$ donc 
$(\lambda_m a^{m-1} + \cdots + \lambda_1) a = 0$ ;  on peut simplifier 
par $a$ qui est régulier d'où $\lambda_m a^{m-1} + \cdots + \lambda_1 = 0$ ;
puis on recommence pour obtenir finalement $\lambda_1 = \cdots = \lambda_m = 0$.
Pour $n \geqslant 2$, on considère $\bfS' = \bfS / \langle a_1 \rangle$.
On a $\bfR \hookrightarrow \bfS'$ car $\bfR \cap \langle a_1 \rangle = \{ 0 \}$.
La suite $\ua' = (\overline{a_2}, \dots, \overline{a_n})$ de $\bfS'$ est régulière et 
$\bfR \cap \langle \ua' \rangle = \{ 0 \}$.
Soit $f \in \bfR[\uX]$ de degré $\leqslant d$ tel que $f(\ua) = 0$.
On écrit $f(\uX) = X_1 q(\uX) + r(\uX')$ avec $q \in \bfR[\uX]$ de degré $\leqslant d-1$ et 
$r \in \bfR[X_2, \dots, X_n]$. 
Dans~$\bfS'$, on a $r(\ua') = 0$ donc par récurrence sur $n$, on obtient $r=0$.
On simplifie par $a_1$ régulier d'où $q(\ua) = 0$.
Par récurrence sur $d$, on obtient $q = 0$ puis $f=0$.
\end{proof}

\begin{proof}[\bf Preuve du théorème~\ref{ProfondeurVeronese}]
Prouvons d'abord que $f(Y_i) \in \scrV_r$ pour tout $i \geqslant r$. 
Pour cela on considère le polynôme $g(T)$ :
$$
g(T) \ =\  
\prod_{i=1}^n (T-Y_i)  
\ =\  T^n - \sigma_1 T^{n-1} + \cdots + (-1)^{n} \sigma_{n}
\ =\  f(T) + h(T)
$$
avec $h(T) = (-1)^r \sigma_r T^{n-r} + \cdots + (-1)^n \sigma_n$. 
Pour $j \geqslant r$, l'élément $\sigma_j$ appartient à $\scrV_r$, 
donc il en est de même de $h(Y_i)$.
Comme $g(Y_i) = 0$, on obtient $f(Y_i) = -h(Y_i)$, d'où le résultat.

Introduisons l'anneau intermédiaire $\bfS$ :
$$
\bfS = \bfR[\sigma_1, \ldots, \sigma_{r-1}, Y_r, \ldots, Y_n] \, \subset\, \bfR[\uY]
$$
Les $f(Y_i)$, pour $r \leqslant i \leqslant n$, appartiennent à cet anneau intermédiaire $\bfS$. 

Montrons que la suite  $\big(f(Y_r), \ldots, f(Y_n)\big)$ est une suite régulière de $\bfS$.
Chaque $f(Y_i)$ est un polynôme unitaire en $Y_i$ à coefficients 
dans $\bfR' = \bfR[\sigma_1, \ldots, \sigma_{r-1}]$.
Ainsi $\big(f(Y_r), \ldots, f(Y_n)\big)$ est une suite régulière de 
$\bfR'[Y_r,\dots, Y_n] = \bfS$ (car $Y_r,\dots, Y_n$ sont des indéterminées sur $\bfR'$
d'après~\ref{LemmeFonctionsSymElem}-ii).

Comme $\bfR[\uY]$ est un $\bfS$-module libre
d'après~\ref{LemmeFonctionsSymElem}-iii, la suite $\big(f(Y_r),
\ldots, f(Y_n)\big)$ est une suite régulière de $\bfR[\uY]$
(d'après~\ref{RegTrivialFacts}-ii).
\end{proof}

\cleardoublepage

\section{Les objets associés au système polynomial $\protect\uP = (P_1,\dots,P_n)$}
\label{ObjetsSuiteP}

Fixons le cadre général de notre étude: un format de degrés 
$D = (d_1,\dots, d_n)$ où les $d_i$ sont des entiers~$\geqslant 1$
et un système polynomial $\uP = (P_1, \dots, P_n)$ où chaque $P_i \in \bfA[\uX] =
\bfA[X_1, \dots, X_n]$ est \emph {homogène} de degré $d_i$.
Le statut de \emph{l'anneau des coefficients} $\bfA$,
variable avec le contexte,
sera précisé à l'occasion. Selon les circonstances,
il arrivera parfois que $\bfA$ soit un anneau de polynômes au dessus
d'un \og petit anneau de base $\bfk$ \fg{}, les indéterminées de
$\bfA/\bfk$ étant allouées aux coefficients des $P_i$; en cas
d'absence de précision, $\bfA$ désignera un anneau commutatif
quelconque. De la même manière, nous préciserons à l'occasion la
nature du système $\uP$ (suite régulière, générique etc.).

\label{NOTA02-D}%
\label{NOTA02-uP}%
\label{NOTA02-bfAX}%
\index{format de degrés}%
\index{anneau des coefficients d'un système}%

\medskip
A la base, il y a les deux systèmes ci-dessous. Celui de gauche est qualifié de
\emph{jeu étalon}. 
Celui de droite, où les $p_i$ sont des indéterminées
sur $\bbZ$, est dit \emph{jeu étalon généralisé}:
$$
\uX^D = (X_1^{d_1}, \dots, X_n^{d_n}), \qquad\qquad
\pXD = (p_1 X_1^{d_1},\dots, p_n X_n^{d_n}) 
$$
Dans la suite, nous allons associer au système $\uP$ un certain nombre
d'objets: son complexe de Koszul, l'idéal $\langle\uP\rangle$, ainsi
que des idéaux monomiaux de $\bfA[\uX]$, divers $\bfA$-modules etc.
Il nous faudra introduire un certain nombre de constructions
et étudier les notions qui y sont attachées (régularité,
propriété de MacRae etc.).  Lorsque les notions introduites ne
dépendent que du format $D$ et pas de $\uP$, 
nous essaierons d'en attirer l'attention du lecteur.

\index{jeu!e@étalon}%
\index{jeu!e@étalon généralisé}%
\label{NOTA02-uXD}%
\label{NOTA02-pXD}%
%
%

\subsection{Système homogène générique, système couvrant le jeu étalon généralisé}

Dans la définition de \emph{système générique} ci-dessous, nous n'imposons pas
l'égalité entre le nombre de polynômes et le nombre d'indéterminées.
Il est à noter que chaque polynôme qui intervient est homogène de
degré spécifié.

\begin{defn}[Système homogène générique]
Soient $n,d$ deux entiers. 


\noindent
Un polynôme $F$ en $n$ indéterminées sur un anneau $\bfk$
est dit générique (homogène) de degré $d$ s'il s'écrit:
$$
F = 
\sum_{\substack{\alpha \in \bbN^n \\ |\alpha| = d}} 
c_\alpha X^\alpha \qquad
\hbox {où les $c_\alpha$ sont des indéterminées distinctes sur $\bfk$}
$$
Une famille de polynômes $(P_1, \dots, P_s)$ est générique sur
$\bfk$ si chaque $P_i$ est générique d'un degré spécifié et si la
famille des indéterminées de $P_i$ (c'est-à-dire des coefficients de $P_i$), est disjointe de celle de $P_j$ pour
$i \ne j$.
\end{defn}

\medskip
En présence d'un format $D = (d_1, \dots, d_n)$, nous considérerons le système générique
$\uP = (P_1, \dots, P_n)$ au-dessus de~$\bfk$ défini par
$$
P_i 
\ = \ 
\sum_{\substack{\alpha \in \bbN^n \\ |\alpha| = d_i}} 
c_{i,\alpha} X^\alpha 
\quad \in \ \bfA[X_1, \dots, X_n]
$$
système qui fait intervenir
$$
\sum_{i=1}^n  \binom {d_i+n-1}{n-1}  \quad
\hbox {indéterminées (distinctes) $c_{i,\alpha}$ sur $\bfk$}
$$
Dans ce contexte, nous écrirons l'anneau des coefficients $\bfA$ sous la
forme $\bfA = \bfk[\indetsPi]$, étant sous-entendu qu'un format $D$
est fixé.

\index{système!générique}%
\label{NOTA02-bfkindetsPi}%

\medskip
Une autre notion, moins classique, est celle de système couvrant un autre système.

\begin{defn}[Système couvrant un autre système à coefficients indéterminées]
\label{JeuCouvrant}
On se donne un système $\uQ$ de format $D$ dont les coefficients sont des
indéterminées sur un anneau $\bfR$.

On dit qu'un système $\uP$ couvre $\uQ$ lorsque, pour tout $i$,
le polynôme $P_i$ a pour coefficient en $X^\alpha$ :
$$
\coeff_{X^\alpha}(P_i) \ = \ 
\left\{
\begin{array}{ll}
\text{un élément de $\bfR$} & \text{si $\coeff_{X^\alpha}(Q_i) = 0$} \\
\coeff_{X^\alpha}(Q_i) & \text{sinon} \\
\end{array}
\right.
$$de sorte que $P_i$ a ses coefficients dans 
$\bfR[\text{\rm indets pour les $Q_i$}]$.
\end{defn}

\index{système!couvrant un autre système}%
\index{couvrir un système}%

\medskip
Considérons par exemple $2n$ indéterminées sur $\bfR$ notées $p_1, \dots,
p_n, q_1, \dots, q_n$ et le système $\uQ$ de format $(d_1, \dots, d_n)$ 
défini par :
$$
Q_i = p_iX_i^{d_i} + q_i X_1^{d_i-1} X_{i+1}  \qquad \qquad \text{en convenant que $X_{n+1} = X_1$}
$$
Par exemple, pour $n= 3$, on a :
$$
\left\{
\begin {array} {ccl}
Q_1 &=& p_1X_1^{d_1} + q_1X_1^{d_1-1} X_2 \\ [0.1cm]
Q_2 &=& p_2X_2^{d_2} + q_2X_1^{d_2-1} X_3 \\ [0.1cm]
Q_3 &=& p_3X_3^{d_3} + q_3X_1^{d_3-1} X_1\\
\end {array}
\right.
$$
Pour fixer les idées, prenons $d_1=2$ de sorte que 
$Q_1 = p_1X_1^2 + q_1X_1X_2$. 
Le polynôme $P_1$ 
d'un système $\uP = (P_1,P_2,P_3)$ couvrant $\uQ$ 
est de la forme :
$$
P_1 = p_1X_1^2 + q_1X_1X_2 + a X_1X_3 + b X_2^2 + c X_2X_3 + d X_3^2, \qquad
\text{avec $a,b,c,d \in \bfR$}
$$
Pour $n$ quelconque, ce système $\uQ$ est de la forme $\pXD +
\bsq\uR^{(1)}$ où $\uR^{(1)}$ est un jeu monomial de format~$D$
étudié dans la section \ref{SectionJeuxSimples} (cf l'exemple
après la proposition \ref{ExistenceSuperHauteur}).  On verra
l'importance de la propriété de couverture, cf. la section
\ref{SectSystemesCreuxGr2} et le théorème \ref{GrViaPcouvreQ}
(contrôle de la profondeur par la propriété de couverture).

\medskip

Vu l'importance qu'aura le jeu étalon généralisé dans notre étude, 
nous reformulons
la \og propriété de couverture\fg{} dans ce cadre.

\begin{defn}[Système couvrant le jeu étalon généralisé]
\label{JeuCouvrantEtalon}
Soit $\bfR$ un anneau quelconque et $p_1, \dots, p_n$ des indéterminées sur $\bfR$. 
On pose $\bfA = \bfR[p_1,\dots, p_n]$. On dit qu'un système $\uP = (P_1, \dots, P_n)$
de polynômes de~$\bfA[\uX]$ couvre le jeu étalon généralisé
$(p_1X_1^{d_1}, \dots, p_n X_n^{d_n})$ lorsque
$$
P_i \ = \ \  
p_i X_i^{d_i} 
\ + \ 
\sum_{\alpha \neq (0, \dots, d_i, \dots, 0)} \kern-8pt
a_{i,\alpha}\, X^\alpha 
\quad 
\text{où $a_{i,\alpha} \in \bfR$}
$$
\end{defn}

Nous abordons maintenant un résultat archi-classique: une suite $\uP$ de $n$
polynômes génériques en $n$ indéterminées est une suite régulière.
C'est encore vrai si on prend $s$ polynômes avec $s \leqslant n$ 
(en revanche, c'est faux si $s > n$, comme on le voit avec $n=1$, $s=2$,
$P_1 = aX$, $P_2 = bX$).
Curieusement, Demazure dit dans~\cite[p.~19]{Demazure1} :
\begin{quote}
\textit{Je n'ai pas de démonstration sympathique à proposer pour cette
proposition du ``folklore''}
\end{quote}
La preuve qui suit, tirée de~\cite[lemme 3.3]{Buse}, utilise
la propriété de permutabilité d'une suite régulière homogène en terrain gradué
appliquée à une suite de polynômes
homogènes de degré $> 0$, cf.~\ref{PermutabiliteSuiteReguliereHomogene}.

\begin {theo} [Une suite de polynômes génériques est une suite régulière] 
\label{SuiteGeneriqueReguliere}

On considère~$s$ polynômes génériques $P_1, \ldots, P_s \in \bfA[\uX]$, homogènes de
degré $d_1, \ldots, d_s \geqslant 1$, 
en $n$ variables $\uX = (X_1, \ldots, X_n)$.
Alors la suite $(P_1, \ldots, P_s)$ est régulière lorsque $s \leqslant n$.
\end {theo}

\begin {proof}

Traitons d'abord le cas particulier $n=s=2$, $d_1=2$, $d_2=3$:
$$
P_1 = aX^2 + bXY + cY^2, \qquad
P_2 = a'X^3 + b'X^2Y + c'XY^2 + d'Y^3
$$
Montrons que la suite:
$$
b,\ c,\ a',\ b',\ c',\ X-a,\ Y-d',\ P_1,\ P_2
$$
est régulière.  
Il est clair que la sous-suite $\uw$ des 7 premiers termes est régulière.  
En ce qui concerne la suite $(\uw, P_1, P_2)$, on peut, 
par définition d'une suite régulière, remplacer $P_i$ par $Q_i$ avec
$Q_i \equiv P_i \bmod \langle\uw\rangle$.
Comme $P_1 \equiv X^3 \bmod \langle \uw\rangle$ et $P_2 \equiv Y^4 \bmod \langle\uw\rangle$, 
il s'agit donc de vérifier que la suite $(\uw, X^3, Y^4)$ est régulière ;
elle l'est, essentiellement car $a$ et $d'$ sont des \textit{indéterminées}.
Ainsi, la suite $(\uw, P_1, P_2)$ est régulière ; de plus, elle est constituée de
polynômes homogènes et reste donc régulière par permutation de ses termes
d'après~\ref{PermutabiliteSuiteReguliereHomogene}.
En conséquence, la suite $(P_1,P_2)$ est régulière.

\smallskip
\noindent
Passons au cas général qui est en fait analogue au cas particulier
ci-dessus. Puisque $s \leqslant n$, on peut isoler dans chaque $P_i$ le
terme $p_i X_i^{d_i}$. 
On considère alors la suite $\uw$ constituée d'une part des coefficients des $P_i$ autres 
que $p_i$ et d'autre part des binômes $X_i - p_i$. 
Il est clair que $\uw$ est une suite régulière et que 
$P_i \equiv X_i^{d_i+1} \bmod \langle\uw\rangle$.  
Puisque la suite $(\uw, X_1^{d_1+1}, \ldots, X_s^{d_s+1})$ est régulière, 
il en est de même de $(\uw, P_1, \ldots, P_s)$ et donc par permutation, 
de $(P_1, \ldots, P_s, \uw)$, a fortiori de $(P_1, \ldots, P_s)$.
\end{proof}

Avec le travail réalisé en~\ref{ControleProfondeur} (plus profond
que la simple permutabilité d'une suite régulière en homogène, cf.~\ref{PermutabiliteSuiteReguliereHomogene}),
on obtient le résultat suivant
(plus général que le précédent dans le sens où la suite générique 
$(P_1, \dots, P_n)$ couvre le jeu étalon généralisé).

\begin{theo} 
\label{JeuCouvrantEtalonRegularite}
Soit $\uP$ un système couvrant le jeu étalon généralisé (cf.~\ref{JeuCouvrantEtalon}).
Alors la suite $\uP$ est régulière.
\end{theo}

\begin{proof}
Voyons $\bfA[\uX]$ comme $\bfR[p_1, \ldots, p_n,X_1,\ldots, X_n]$. 
Dans cet anneau de polynômes, $P_i$ est de degré $1+d_i$ et sa
composante homogène dominante (de degré $1+d_i$) est $p_iX_i^{d_i}$. 
La suite $(p_1X_1^{d_1}, \ldots, p_nX_n^{d_n})$ est régulière, donc
d'après la propriété~\ref{ControleProfondeur}-iii), 
la suite $\uP$ est régulière.
\end{proof}

\subsection{Matrice bezoutienne d'un jeu $\protect\uP$ et son déterminant $\nabla$}

\centerline{
\fbox{\parbox{0.95\linewidth}{
Dorénavant $\uP = (P_1, \dots, P_n)$ désigne un système homogène de $\bfA[\uX] = \bfA[X_1, \dots, X_n]$ 
de format de degrés $D = (d_1, \dots, d_n)$}
}}

\medskip
\begin {defn} [Degré critique]
Le degré critique du format $D$ ou du système $\uP$ est l'entier $\delta \geqslant 0$ défini par
$\delta = \sum_{i=1}^n (d_i-1)$.
\end{defn}

Ce degré critique doit être perçu comme la différence entre les degrés des deux systèmes $\uP$ et $\uX$
au sens où $\delta = \sum_i\deg P_i - \sum_i\deg X_i$. Un monôme $X^\alpha$ de degré $\delta$
est divisible par au moins un $X_i^{d_i}$ sauf si $\alpha = (d_1-1, \dots, d_n-1)$. En ce degré $\delta$,
on a donc l'égalité 
$$
\bfA[\uX]_\delta \ = \ 
\langle X_1^{d_1}, \dots, X_n^{d_n} \rangle_\delta \ \oplus \ \bfA X^{\emouton} 
$$
où le monôme $X^{\emouton} = X_1^{d_1 - 1}\cdots X_n^{d_n - 1}$ est baptisé mouton-noir \mouton.

\label{NOTA02-delta}%
\label{NOTA02-Xemouton}%
\index{degré critique d'un système}%
\index{MoutonNoir@{\textit{mouton-noir}}}%
\index{format de degrés}%

\bigskip 
Le fait que chaque polynôme $P_j$ soit homogène de degré $d_j$ permet
d'écrire de manière non unique:
$$
P_j = \sum_{i=1}^n X_i \dsV_{ij}
\qquad 
\text{où $\dsV_{ij}$ est homogène de degré $d_j-1$}
$$
Cette égalité est une façon de \og passer \fg{} de la suite $\uX$ à la suite $\uP$, 
au sens matriciel suivant $\uP = \uX \dsV$. 
Cette écriture conduit à la notion de
matrice bezoutienne et déterminant bezoutien de $\uP$.
Le fait que la suite $\uX$ soit régulière, a fortiori
1-sécante, va jouer un rôle important. 
On rappelle, cf.~\ref{DefSuite1Secante}, 
que le caractère 1-sécant de $\uX$ signifie simplement qu'en
combinant $\bfA[\uX]$-linéairement les relations $X_i\times X_j - X_j
\times X_i = 0$, on obtient toutes les relations linéaires entre 
les~$X_i$.  De manière plus formelle, tout relateur $\bfA[\uX]$-linéaire
entre les~$X_i$ est combinaison $\bfA[\uX]$-linéaire des
$X_ie_j - X_je_i$.

\begin{defn}\label{DefNabla}
Soit $\dsV$ une matrice homogène transformant $\uX$ en $\uP$ au sens suivant:
$$
\begin{bmatrix}
P_1 & \cdots &  P_n
\end{bmatrix} 
\ = \ 
\begin{bmatrix}
X_1 & \cdots &  X_n
\end{bmatrix} \dsV
$$
Ici, matrice homogène signifie que $\dsV_{ij}$ est un polynôme homogène de degré $d_j-1$ en $\uX$.
Une telle matrice est appelée \textit{matrice bezoutienne du jeu $\uP$}.

Le déterminant de $\dsV$, noté $\nabla = \nabla(\uX)$, est un polynôme
homogène de degré $\delta$.  Ce polynôme $\nabla \in \bfA[\uX]_\delta$
sera nommé \og déterminant bezoutien de $\uP$ \fg{}.
\end{defn}

\label{NOTA02-nabla}%
%
\index{matrice!bezoutienne1@$\uX$-bezoutienne}%
\index{de@déterminant bezoutien}%

En cas de conflit de terminologie,
nous utiliserons \og $\uX$-bezoutien \fg{} en lieu et place de \og bezoutien\fg{}. 
On retiendra dans cette définition que la colonne $j$ de $\dsV$ doit être homogène de degré $d_j-1$. 
N'importe quelle matrice $\dsV$ vérifiant $\uP = \uX \, \dsV$ n'est pas 
homogène au sens ci-dessus (bien que les suites $\uP$ et $\uX$ soient homogènes au sens habituel). 
Ainsi, à partir d'une matrice homogène $\dsV$, 
on peut très bien ajouter à sa première colonne le vecteur 
${}^{\rm t}(-X_2^e,\,X_1X_2^{e-1},\, 0, \dots, 0)$, 
conduisant à une matrice $\dsV'$ vérifiant $\uP = \uX \, \dsV'$ dont la première colonne 
n'est pas homogène de degré $d_1 - 1$.


\medskip
Un exemple simple: la matrice diagonale $\diag(p_1X_1^{d_1-1},
\dots, p_nX_n^{d_n-1})$ est une matrice bezoutienne du jeu étalon
généralisé $\pXD$ et son déterminant~$\nabla$ est $p_1 \cdots p_n
X^\emouton$, multiple du monôme \MoutonNoir.

\medskip

\index{matrice!bezoutienne1@$\uX$-bezoutienne!rigidifiée}%

Il existe de nombreuses manières de construire des matrices
bezoutiennes d'un système $\uP$.  En voici une que nous qualifions
parfois de version \og rigidifiée\fg{} (qui suit la version rigidifiée
en 2 jeux de variables, cf. définition
ultérieure~\ref{DefBezoutien})
$$
\dsV_{ij} =
\frac{P_j(0,\dots,0,\cercle{$X_i$},X_{i+1}, \dots, X_n) -
P_j(0,\dots,0 ,\cercle{$0$},X_{i+1}, \dots, X_n)} {X_i}
$$
Avec cette définition, le polynôme $\dsV_{ij}$ appartient à $\bfA[X_i,
  \ldots, X_n]$.  Ainsi, la dernière ligne de $\dsV$ ne dépend que de
$X_n$, l'avant-dernière ligne de $X_{n-1}, X_n$ et ainsi de suite.  Si
l'on impose cette contrainte, à savoir $\dsV_{i,j} \in \bfA[X_i, \ldots, X_n]$,
alors la matrice $\dsV$ ci-dessus est unique à vérifier $\uP = \uX
\,\dsV$.  Ceci vient du fait que pour $F \in \langle X_1, \ldots,
X_n\rangle$, il n'y a qu'une seule manière de l'écrire $F = U_1 X_1 +
\cdots + U_n X_n$ avec $U_i \in \bfA[X_i, \ldots, X_n]$.  Pour assurer
cette unicité (\idest{} $U_1 X_1 + \cdots + U_n X_n = 0$ avec $U_i \in
\bfA[X_i, \ldots,X_n]$ entraîne $U_1 = \dots = U_n = 0$), il suffit de
réaliser la spécialisation $X_1, \ldots, X_{n-1} := 0$, ce qui fournit
$U_n = 0$ (car $U_n$ ne dépend pas de $X_1, \dots, X_{n-1}$), puis
$X_1, \ldots, X_{n-2} := 0$, ce qui fournit $U_{n-1} = 0$
(car $U_{n-1}$ ne dépend pas de $X_1, \dots, X_{n-2}$) et ainsi de suite.  
On notera que l'existence et l'unicité de l'écriture
de $F$ repose essentiellement sur une application répétée
de la décomposition $\bfR[T] = T\bfR[T] \oplus \bfR$, 
où $T$ est une indéterminée sur un anneau~$\bfR$.

\bigskip
Voici un exemple pour le jeu $P_i = S_iX_i^{d_i-1}$ avec 
$S_i = X_i + X_{i+1}$ (la notation est circulaire \idest{} $X_{n+1}$ désigne $X_1$). 
Ci-dessous, deux matrices bezoutiennes de $\uP$, schématisées pour $n = 4$ ;
une matrice $\dsV'$ et une matrice $\dsV''$ qui, elle, fonctionne uniquement pour $d_i \geqslant 2$
$$
\dsV' \ = \
\begin {bmatrix}
X_1^{d_1-1} &.           &.         &X_4^{d_4-1} \\
\noalign {\medskip}
X_1^{d_1-1} &X_2^{d_2-1}  &.          &. \\
\noalign {\medskip}
.          &X_2^{d_2-1}  &X_3^{d_3-1} &. \\
\noalign {\medskip}
.          &.           &X_3^{d_3-1} &X_4^{d_4-1} \\
\end {bmatrix}
\qquad\qquad
\dsV'' \ = \
\begin {bmatrix}
S_1X_1^{d_1-2} &.           &.         &. \\
\noalign {\medskip}
.                   &S_2X_2^{d_2-2}  &.          &. \\
\noalign {\medskip}
.                   &.            &S_3X_3^{d_3-2} &. \\
\noalign {\medskip}
.                   &.           &.                 &S_4X_4^{d_4-2} \\
\end {bmatrix}
$$
Que donne la version \og rigidifiée \fg{} $\dsV$ ?
La ligne $i$ de $\dsV$ ne comporte que des indéterminées d'indice supérieur à $i$.
Ainsi, la colonne $j$ de $\dsV$, notée $\dsV_j$, vérifie 
$$
\forall\, j \leqslant n-1,\quad 
\dsV_j \ = \ 
\left \{
\begin{array}{ll}
\dsV''_j & \text{si $d_j \geqslant 2$} \\[0.2cm]
\dsV'_j & \text{si $d_j = 1$} \\
\end{array}
\right .
\qquad \text{et } \qquad 
\dsV_n = \dsV'_n
$$
Par exemple, pour les formats 
$(2,4,3)$, $(2,4,1)$ et $(2,1,3)$, voici 
la matrice bezoutienne rigidifiée $\dsV$ :
$$
\left[
\begin{array}{*{3}{c}}
X_{1}+X_{2}&  . & X_{3}^{2} \\ 
 . & (X_{2} + X_{3})X_2^2 &  .  \\ 
 . &  . & X_{3}^{2} \\ 
\end{array}
\right]
\quad
\left[
\begin{array}{*{3}{c}}
X_{1}+X_{2}&  . & 1 \\ 
 . & (X_{2} + X_{3})X_2^2 &  .  \\ 
 . &  . & 1 \\ 
\end{array}
\right]
\quad
\left[
\begin{array}{*{3}{c}}
X_{1}+X_{2}&  . & X_{3}^{2} \\ 
 . & 1 &  .  \\ 
 . &  1 & X_{3}^{2} \\ 
\end{array}
\right]
$$

\begin {defn} [L'idéal saturé de $\langle\uP\rangle$ et l'idéal d'élimination] \leavevmode

Au système $\uP$ sont associés deux idéaux: un idéal de $\bfA[\uX]$ dit idéal saturé de $\langle\uP\rangle$
$$
\uPsat \ = \ 
\Big \{
F \in \bfA[\uX] \mid \exists\, e \in \bbN, \ 
\forall\, i,\ X_i^e \,F  \in \langle \uP \rangle
\Big\}
$$
ainsi qu'un idéal de $\bfA$,  dit idéal d'élimination de $\uP$, défini par $\ElimIdeal$.
\end {defn}

\index{idéal!saturé de}%
\index{idéal!d'élimination}%
\label{NOTA02-uPsat}%
\label{NOTA02-ElimIdeal}%

Pour $\ElimIdeal$, la terminologie utilisée par Demazure dans
\cite{Demazure2} est ``idéal éliminant de $\uP$''.  
L'étude de cet idéal de $\bfA$ 
(composante homogène de degré $0$ de $\uPsat$) passe par l'étude 
de toutes les composantes homogènes de $\uPsat$ ; par exemple, 
celle de degré $\delta$ contient tout déterminant bezoutien $\nabla$ comme on le voit
dans la proposition qui suit.

\medskip

Précisons que notre point de vue est essentiellement algébrique et que
notre travail ne contient pratiquement pas de considérations
géométriques.  Cela signifie, pour un système homogène $\uP$ à
coefficients dans un anneau $\bfA$, que l'étude algébrique de l'idéal saturé
$\uPsat$ ou de l'idéal d'élimination $\uPsat \cap \bfA$ est un
objectif en soi.


\noindent
Au dessus d'un corps $\bfk$, on peut juste signaler le fait banal suivant, en liaison avec l'idéal
d'élimination, en nous plaçant dans le cadre que nous avons retenu
(autant d'indéterminées que de polynômes).  Supposons disposer de $n$
polynômes homogènes $P_1, \ldots, P_n \in \bfk[X_1, \ldots, X_n]$,
ayant un zéro commun $\uxi \in \bbP^{n-1}(\bfk')$ sur une extension $\bfk'$ de $\bfk$.
Alors l'idéal
d'élimination de $\uP$ est réduit à $0$; en effet, dire que
$a\in \uPsat\cap\bfA$ équivaut à dire que pour un $e$ assez grand, on
a $aX_i^e \in \langle\uP\rangle$ pour tout $i$ et il suffit alors de
spécialiser $\uX$ en~$\uxi$, pour obtenir $a=0$.

\begin{prop}[Propriétés du déterminant bezoutien] 
\label{NablaDansLeSature}
Soit $\uP$ un jeu quelconque.

\begin{enumerate}[\rm i)]
\item 
Pour tout $1 \leqslant i \leqslant n$, 
on a $X_i \nabla \in \langle \uP \rangle_{\delta+1}$. En conséquence, $\nabla \in \uPsat_\delta$.

\item 
Pour deux déterminants bezoutiens $\nabla$, $\nabla'$ de $\uP$, on a
$\nabla - \nabla' \in \uPdelta$.

\item 
Pour tout $d \ne \delta$, on a 
$\langle \uP, \, \nabla \rangle_d = \langle \uP \rangle_d$.
\end{enumerate}
\end{prop}

\begin{proof}\leavevmode

\noindent  
i) Il suffit de multiplier à droite l'égalité $\uP = \uX\,\dsV$ par la
cotransposée de~$\dsV$.

\noindent
ii) Cela résulte du fait que la suite $\uX$ est 1-sécante et de~\ref{IndependanceNabla}.

\noindent
iii) Par définition, on a  $\langle \uP, \, \nabla \rangle_d = 
\langle \uP \rangle_d + \bfA[\uX]_{d-\delta} \nabla$.
Si $d < \delta$, on a $\bfA[\uX]_{d-\delta} = 0$ d'où l'égalité annoncée.
Pour $d - \delta \geqslant 1$, puisque $X_i \nabla \in \langle \uP \rangle$,
on a l'inclusion $\bfA[\uX]_{d-\delta} \nabla \subset \langle \uP \rangle_d$
d'où encore l'égalité annoncée.
\end{proof}

Dans les chapitres à venir, cf.~\ref{MacRaeJeuEtalonDelta}, 
\ref{sectionFormeLineaireOmega}, \ref{PoidsNormalisationMacRae}, 
nous associerons à tout jeu $\uP$ de degré critique $\delta$ plusieurs formes linéaires
$\bfA[\uX]_\delta \to \bfA$, dont deux bien particulières notées
$$
\omega = \omega_\uP
\qquad \text{ et } \qquad  
\omegares = \omegaresP
$$
Les évaluations de ces formes en un déterminant bezoutien $\nabla$ de $\uP$ jouent
un rôle capital dans notre étude; par exemple $\omegaresP(\nabla)$ n'est autre que
le résultant de $\uP$!

\medskip
Ici, pour les deux jeux particuliers 
$\uX^D = (X_1^{d_1}, \dots, X_n^{d_n})$ et 
$\pXD = (p_1X_1^{d_1}, \dots, p_nX_n^{d_n})$, nous
définissons ces formes linéaires de manière artificielle, de façon
à pouvoir considérer rapidement leur action sur les déterminants bezoutiens
respectifs.
\`A cet
effet, nous introduisons la forme coordonnée sur le \MoutonNoir{}
relativement à la base monomiale de $\bfA[\uX]_\delta$, forme linéaire que nous
notons $(X^{\emouton})^\star : \bfA[\uX]_\delta \to \bfA$.
Et nous posons:
$$
\omega_{\uX^D} = (X^{\emouton})^\star, \qquad\qquad
\omegarespXD 
= p_1^{\widehat d_1-1}\cdots p_n^{\widehat d_n-1}(X^{\emouton})^\star
\qquad \hbox {avec} \quad
\widehat d_i = \prod_{j \ne i} d_j
$$

\label{NOTA02-hatdi}%

\begin{prop}[Matrice bezoutienne du jeu étalon et jeu étalon généralisé]
\label{NablaEtalon} \leavevmode

\begin{enumerate}[\rm i)]
\item 
La matrice diagonale $\diag(X_1^{d_1-1}, \dots, X_n^{d_n-1})$ 
est une matrice bezoutienne du jeu étalon $\uX^D$ 
et son déterminant~$\nabla$ est le monôme {\rm mouton-noir} $X^\emouton$.
Ainsi $\omega_{\uX^D}(\nabla) = 1$.

\item 
La matrice diagonale $\diag(p_1X_1^{d_1-1}, \dots, p_nX_n^{d_n-1})$ 
est une matrice bezoutienne du jeu étalon généralisé $\pXD$,  de
déterminant~$\nabla$ égal à $p_1 \cdots p_n X^\emouton$.
Et on a 
$\omegarespXD(\nabla) = p_1^{\widehat d_1} \cdots p_n^{\widehat d_n}$.
 
\item 
On a $\Ker \omega_{\uX^D} = \Ker\omegarespXD = \langle \uX^D \rangle_\delta$.

Par conséquent, pour tout déterminant bezoutien $\nabla'$ de $\uX^D$, on a 
$\omega_{\uX^D}(\nabla') = 1$. Et pour tout déterminant bezoutien $\nabla'$ de $\pXD$, on a 
$\omegarespXD(\nabla') = p_1^{\widehat d_1} \cdots p_n^{\widehat d_n}$.
\end{enumerate}
\end{prop}

\begin{proof} \leavevmode

\noindent
Les points i) et ii) sont immédiats.

\noindent
iii) Par définition, le supplémentaire monomial de
$\bfA X^{\emouton}$ dans $\bfA[\uX]_\delta$ est 
$\langle \uX^D \rangle_\delta = \langle X_1^{d_1}, \dots, X_n^{d_n} \rangle_\delta$, 
d'où 
$$
\bfA[\uX]_\delta = \langle\uX^D\rangle_\delta \oplus \bfA X^{\emouton}
\qquad
\text{ et } 
\qquad
\Ker (X^{\emouton})^\star = \langle\uX^D\rangle_\delta
$$
Les formes linéaires  $\omega_{\uX^D}$ et $\omegarespXD$ sont multiples
de $(X^{\emouton})^\star$ (via un scalaire régulier) et ont donc même noyau.

Pour le \og par conséquent \fg{}, on remarque que chacun des deux
systèmes $\uP$ vérifie
$\uPdelta \subset \langle \uX^D \rangle_\delta$. Par ailleurs, le 
point ii) de la proposition précédente fournit $\nabla
- \nabla' \in \langle \uX^D \rangle_\delta$. 
Chacune des deux formes
linéaires étant nulle sur $\langle \uX^D \rangle_\delta$ a même
valeur en $\nabla$ et $\nabla'$.
\end{proof}

\subsection{Les quotients $\bfB=\bfA[\protect\uX]/\langle\protect\uP\rangle$ et
            $\bfB'=\bfA[\protect\uX]/\langle\protect\uP,\nabla\rangle$} 
\label{DefQuotientsP}

Comme $\uP$ est une suite de polynômes \textit{homogènes}, le quotient
$\bfB = \bfA[\uX]/\langle \uP \rangle$ est un $\bfA[\uX]$-module
gradué et ses composantes homogènes vont jouer un rôle important dans la
suite.  
Le contexte sera assez clair pour nous permettre d'utiliser une notation pour 
ce quotient ne faisant pas référence explicitement à $\uP$. 
Bien entendu, dans les passages qui le nécessitent,
nous rappellerons qui est $\uP$, précisant ainsi $\bfB$.

\label{NOTA02-bfB}%

\medskip
Ce $\bfA[\uX]$-module gradué $\bfB$ est par définition de présentation
finie, présenté par l'application de Sylvester, 
forme linéaire sur $\bfA[\uX]^n$ canoniquement associée à 
$(P_1, \dots, P_n)$:
$$
\Syl = \Syl(\uP) : \, 
\xymatrix @C=1.5cm @M=0.4pc{
\bfA[\uX]^n \ar[r]^-{
\left [
\begin{smallmatrix}
P_1 & \cdots & P_n
\end{smallmatrix}
\right]
} & \bfA[\uX] 
}
$$
Pour chaque $d$, la composante homogène $\bfB_d$ de degré $d$ est
un $\bfA$-module de présentation finie, présenté par l'application de
Sylvester de degré $d$ :
$$
\Syl_d : 
\begin{array}[t]{rcl}
\displaystyle \bigoplus_{i=1}^n \bfA[\uX]_{d - d_i}\,e_i & \longrightarrow & \bfA[\uX]_d \vspace{0.1cm} \\
U_1 e_1 + \cdots + U_n e_n & \longmapsto & \displaystyle U_1P_1 + \cdots +  U_nP_n
\end{array}
$$

\label{NOTA02-Syl}%
\label{NOTA02-Syld}%
\index{application de Sylvester}%

\bigskip
Un second $\bfA[\uX]$-module interviendra également: il s'agit de
$\bfB' = {\bfB}/{\langle \overline \nabla \rangle} = \bfA[\uX]/\langle \uP, \nabla \rangle$.
Il ne dépend que de $\uP$ et pas de $\nabla$ puisque
pour deux déterminants bezoutiens $\nabla$, $\nabla'$ 
de~$\uP$, on a $\nabla - \nabla' \in \langle\uP\rangle$, 
cf.~\ref{NablaDansLeSature}-ii).
C'est un module de présentation finie, présenté par :
$$
\xymatrix @C=2.5cm @M=0.4pc{
\bfA[\uX]^n \oplus \bfA[\uX] \ar[r]^-{
\left [
\begin{smallmatrix}
P_1 & \cdots & P_n & \nabla
\end{smallmatrix}
\right]
} & \bfA[\uX] 
}
$$

\label{NOTA02-bfBprime}%

En ce qui concerne les composantes homogènes de degré $d$,
la seule diffèrence entre $\bfB$ et $\bfB'$ réside dans la composante de degré $\delta$ :
$$
\bfB'_\delta \ = \ 
\frac{\bfB_\delta}{\bfA \overline \nabla}
\ = \ 
\frac{\bfA[\uX]_\delta}{\uPdelta + \bfA \nabla}
\qquad \text{ et } \qquad 
\forall\, d \ne \delta, \quad 
\bfB'_d = \bfB_d
$$
Le $\bfA$-module $\bfB'_\delta$ est présenté par $\Syl_\delta \oplus \rm{mult}_\nabla$ 
$$
\Syl_{\delta} \oplus \rm{mult}_\nabla : \  
\begin{array}[t]{rcl}
\displaystyle \bigoplus_{i=1}^n \bfA[\uX]_{\delta - d_i}\,e_i  \oplus \bfA 
& \longrightarrow & \bfA[\uX]_\delta \vspace{0.1cm} \\
U_1 e_1 + \cdots + U_n e_n \oplus u
& \longmapsto & 
\displaystyle U_1P_1 + \cdots +  U_nP_n \ + \ u \nabla
\end{array}
$$

\subsection{Composantes homogènes du saturé~$\langle\protect\uP \rangle^\sat$ et
            annulateur de $\bfB_d = \bfA[\protect\uX]_d/\langle\protect\uP\rangle_d$}
\label{SectCompHmgSature}

Pour un déterminant bezoutien $\nabla$ de $\uP$, nous avons vu en~\ref{NablaDansLeSature} que
$\langle\uX\rangle \nabla \subset \langle\uP\rangle$.  Peut-on \og
transporter\fg{} $\langle\uX\rangle$ dans $\langle\uP\rangle$, ou bien
$\nabla$ dans $\langle\uP\rangle$ autrement?  Si $\uP$ est régulière,
le théorème de Wiebe (cf.~\ref{WiebeTheorem}) affirme que non. Précisément :
\begin{quote}
\textbf{Théorème de Wiebe dans le contexte polynomial.} 
\it Pour $\uP$ régulière, on a l'égalité des idéaux transporteurs de $\bfA[\uX]$ 
$$
\big(\langle\uP\rangle : \nabla\big) = \langle \uX \rangle 
\qquad \text{et}\qquad 
\big(\langle\uP\rangle : \uX\big) = \langle \uP,\, \nabla\rangle
$$
\end{quote}
De ces deux égalités ayant lieu dans $\bfA[\uX]$, on va extraire d'autres égalités au niveau de $\bfA$ 
en prenant les composantes homogènes,
notamment pour expliciter certaines composantes homogènes 
du saturé de~$\langle\uP\rangle$ :
$$
\uPsat \ = \ 
\Big \{
F \in \bfA[\uX] \mid \exists\, m \in \bbN, \ 
\forall\, i,\ X_i^m \,F  \in \langle \uP \rangle
\Big\}
$$
On a vu par exemple que tout déterminant bezoutien $\nabla$ appartient à $\uPsat_\delta$.

\index{idéal!saturé de}%

\begin{prop}[Composante homogène $\uPsat_d$ en degré $d \geqslant \delta$] \label{MiniWiebe}
Supposons $\uP$ régulière. 

\begin{enumerate}[\rm i)]
\item
Pour $a \in \bfA$, on a l'implication $a\nabla \in \uPdelta
\ \Rightarrow\ a = 0$, traduisant la nullité du $\bfA$-module
$\big(\langle\uP\rangle : \nabla\big)_0$, ce dernier $\bfA$-module étant encore égal à 
$\Ann(\overline\nabla)$ où $\overline\nabla$ est la classe de $\nabla$
dans $\bfB_\delta$.

\item 
Pour tout $d \geqslant \delta$, on a 
$\uPsat_d \, = \, \langle \uP,\, \nabla \rangle_d$, égalité qui se scinde en:
$$
\uPsat_\delta \, = \,\uPdelta \, \oplus \, \bfA \nabla
\qquad \text{ et } \qquad 
\forall\, d \geqslant \delta+1,\  
\uPsat_d \, = \, \langle \uP \rangle_d
$$
\end{enumerate}
\end{prop}

\begin{proof}
i) 
Considérons l'égalité des idéaux homogènes
$\big(\langle\uP\rangle : \nabla\big) = \langle \uX \rangle$.
Prenons-en la composante homogène de degré $0$.
Comme $\langle \uX \rangle_0$ est réduit à 0, on obtient $\big(\langle\uP\rangle : \nabla\big)_0 = 0$.

\medskip
ii) 
%
L'inclusion $\uPsat_d \supset \langle \uP,\, \nabla \rangle_d$ est évidente 
(en effet, les polynômes $P_i$ et $\nabla$ sont dans $\uPsat$).
Pour l'autre inclusion, on procède par récurrence.
Afin d'alléger les notations et de simplifier la lecture, 
nous reportons au dernier moment, autant que faire se peut, 
la considération du degré (ou des composantes homogènes).
Pour $e\in\bbN$, montrons par récurrence
la propriété suivante que nous notons $\calH_e$:
$$
\calH_e \ : \qquad \qquad 
\forall\, F \in \bfA[\uX]_{\geqslant \delta},\quad 
\langle \uX \rangle^e F \subset \langle \uP \rangle 
\ \implies \ 
F \in \langle \uP,\, \nabla\rangle
$$
Pour $e = 0$, c'est évident.
Supposons $\calH_e$ vraie.
Considérons $F \in \bfA[\uX]_{\geqslant \delta}$ tel que 
${\langle \uX \rangle^{e+1} F \subset \langle \uP \rangle}$.
En écrivant $\langle \uX \rangle^{e+1} = \langle \uX \rangle^{e} \langle \uX \rangle$ 
et en utilisant l'hypothèse $\calH_e$ à tous les polynômes $X_i F$ 
(qui sont bien dans $\bfA[\uX]_{\geqslant \delta}$), 
on obtient que $X_i F \in  \langle \uP,\, \nabla\rangle$.
En considérant les composantes homogènes de degré $\geqslant \delta+1$ 
et en utilisant que $\langle \uP,\, \nabla\rangle_{\geqslant \delta+1} = 
\langle \uP \rangle_{\geqslant \delta+1}$ (cf.~\ref{NablaDansLeSature}-iii), 
on obtient $X_i F \in \langle \uP \rangle$ pour tout $i$.
Donc~$F$ appartient à $\big(\langle\uP\rangle : \uX\big)$ qui vaut $\langle \uP,\, \nabla\rangle$ 
d'après le théorème de Wiebe. D'où $\calH_{e+1}$.

\medskip
On a donc démontré que pour tout $d \geqslant \delta$, 
on a $\uPsat_d \, = \, \langle \uP,\, \nabla \rangle_d$.
Or 
$$
\langle \uP,\, \nabla \rangle_d
\ = \ 
\left\{
\begin{array}{ll}
\uPdelta \, + \, \bfA \nabla & \text{car $\nabla$ est de degré $\delta$} \\[0.2cm]  
\langle \uP \rangle_d & \text{pour $d \geqslant \delta+1$, 
d'après \ref{NablaDansLeSature}-iii)} \\
\end{array}
\right.
$$
La somme en degré $\delta$ est directe d'après le point i).
D'où le résultat.
\end{proof}

\medskip
En conséquence, pour un polynôme de degré $\delta$ dans le saturé $\uPsat$,
on maîtrise un exposant $m$ intervenant dans la définition 
(précédant la proposition) du saturé, puisque l'on peut prendre $m=1$.
En effet, un polynôme $F \in \uPsat_\delta$ est, d'après la proposition précédente, 
une combinaison $\bfA$-linéaire de polynômes de $\uPdelta$ et de $\nabla$.
Et comme $X_i \nabla \in \langle\uP\rangle$, on en déduit que 
$X_i F \in \langle\uP\rangle$ pour tout $i$, d'où l'exposant $m=1$.

\begin{prop}[Idéal d'élimination $\ElimIdeal$ et divers annulateurs]
\label{AnnEqualities}
\leavevmode

\begin{enumerate}[\rm i)]
\item 
Pour une suite $\uP$ quelconque, on dispose des inclusions suivantes,
la réunion en $d$ étant croissante:
$$
\left\{
\begin {array} {ll}
\forall\ e \geqslant 1, &\quad
\displaystyle {\Ann(\bfB_e) \ \subset \ \ElimIdeal \ \subset\ \bigcup_{d\geqslant 1} \Ann(\bfB_d)}
\\ [0.5cm]
\forall\ d \geqslant \delta + 1, &\quad
  \Ann(\bfB'_\delta)  \ \subset\ \Ann(\bfB_d) \ \subset\  \ElimIdeal
\end {array}
\right.
$$

\item 
Si de plus $\uP$ est régulière, on a les égalités:
$$
\Ann(\bfB'_{\delta}) = \ElimIdeal
\qquad \text{ et } \qquad 
\forall\, d \geqslant \delta+1, \quad 
\Ann(\bfB_d) = \ElimIdeal
$$
Ce qui se résume en : pour $d \geqslant \delta$, tous les idéaux
annulateurs $\Ann(\bfB'_d)$ sont égaux à $\ElimIdeal$.
\end{enumerate}
\end{prop}

\begin{proof}\leavevmode

\noindent  
 i) Croissance en $d$. Soit $a \in \Ann(\bfB_d)$. On a alors  
$aX^\alpha \in \langle\uP\rangle$ pour tout $|\alpha| = d$ ; a fortiori,
$aX_iX^\alpha \in \langle\uP\rangle$, pour $1\leqslant i \leqslant n$. Or tout
$X^\beta$ de degré $d+1$ s'écrit $X^\beta = X_i X^\alpha$ pour un
certain $i$ et $|\alpha| = d$, ce qui prouve $a \in \Ann(\bfB_{d+1})$.

\smallskip
\noindent
Montrons que $\Ann(\bfB'_\delta) \subset \Ann(\bfB_{\delta+1})$.
Soit $a \in \Ann(\bfB'_\delta)$. Par définition, pour $|\alpha| = \delta$, on a 
$a X^\alpha \in \langle \uP, \nabla \rangle_\delta$; donc
$aX_iX^\alpha \in \langle\uP\rangle$, pour $1\leqslant i \leqslant n$. Or tout
$X^\beta$ de degré $\delta+1$ s'écrit $X^\beta = X_i X^\alpha$ pour un
certain $i$ et $|\alpha| = \delta$, ce qui prouve $a \in \Ann(\bfB_{\delta+1})$.

\smallskip
\noindent
Les autres inclusions, ne présentant pas plus de difficulté, sont laissées au lecteur.

\medskip
\noindent
ii) Il suffit de voir que $\ElimIdeal \subset \Ann(\bfB'_\delta)$.  Prenons $a \in \ElimIdeal$.
Pour tout $|\alpha| = \delta$, on a $aX^\alpha \in \uPsat_\delta$. Or $\uPsat_\delta$
est égal à $\langle \uP,\, \nabla \rangle_\delta$ d'après~\ref{MiniWiebe}.
Ainsi, $aX^\alpha \in \langle \uP,\, \nabla \rangle_\delta$, prouvant que $a \in \Ann(\bfB'_\delta)$.
\end{proof}

\subsubsection*{En dessous du degré critique $\delta$}

\begin {prop} [L'aspect \og non-dégénéré\fg{} de la multiplication lorsque $\uP$ est régulière]
\leavevmode
\label {SemiPairingProperty}

Soit $\uP$ une suite \emph{régulière} et $d,d'$ vérifiant $d+d' = \delta$.

\begin {enumerate} [\rm i)]
\item
Pour $F \in \bfA[\uX]_d$, on a l'implication:
$$
F\bfA[\uX]_{d'} \subset \uPdelta   \quad\Longrightarrow\quad  F \in \langle\uP\rangle_d
$$
\item
La multiplication $\bfB_d \times \bfB_{d'} \to \bfB_\delta$ est \og non-dégénérée\fg{} au sens où, pour $b \in \bfB_d$:
$$
b\bfB_{d'} = 0   \quad\Longrightarrow\quad  b = 0
$$
Ou encore, l'application $\bfA$-linéaire canonique $\bfB_d \to \Hom_\bfA(\bfB_{d'}, \bfB_\delta)$ est injective.
\end {enumerate}
\end {prop}

\begin {proof} \leavevmode

Le second point est une redite du premier, premier que nous montrons
par récurrence sur $d'$. Pour $d'=0$, c'est évident. Supposons $d' \geqslant 
1$. On a $X_iF \in \bfA[\uX]_{d+1}$ et $X_iF\bfA[\uX]_{d'-1} \subset
\uPdelta$, donc par récurrence, $X_iF \in \langle\uP\rangle_{d+1}$.
Ainsi $F$ appartient au transporteur $(\langle\uP\rangle : \uX)$.
D'après le théorème de Wiebe (formulation citée en début de section), on a $F \in \langle\uP,\nabla\rangle$
donc $F \in \langle\uP,\nabla\rangle_d$, égal à $\langle\uP\rangle_d$ puisque $d < \delta$.
\end {proof}

\subsubsection*{Reformulation en cohomologie de \Cech}

Par définition, pour un $\bfA[\uX]$-module $M$, son module de
cohomologie de \Cech{} (relativement à la suite~$\uX$) de degré
homologique $0$ est défini par:
$$
\vH^0_{\uX}(M) = \{ m \in M \hbox { annulé par une puissance de l'idéal $\langle\uX\rangle$} \}
$$
En particulier, pour $M = \bfA[\uX]/\langle\uP\rangle$ que nous avons noté $\bfB$, on a :
$$
\vH^0_{\uX}(\bfB)  = \uPsat /\langle\uP\rangle
$$

\index{cohomologie de \Cech}%
\label{NOTA02-H0Cech}%

\begin{prop} \label{WiebeCech}
Supposons $\uP$ régulière. 

\begin{enumerate}[\rm i)]

\item
 Le déterminant bezoutien $\overline\nabla$ de $\uP$  est un élément sans
 torsion du $\bfA$-module $\bfB_\delta$.

\item 
Le module de cohomologie de \Cech{}, relativement à la suite~$\uX$, de degré homologique $0$,
du $\bfA[\uX]$-module $\bfB$ 
a pour composante homogène :
$$
\vH^0_{\uX}(\bfB)_d 
\ \overset{\rm def}{=} \ 
\bigg(
\dfrac{\uPsat}{\langle\uP\rangle}
\bigg)_d
\ = \ 
\left \{
\begin{array}{lr}
\bfA \overline \nabla & \text{si $d = \delta$} \\[0.3em]
0 & \text{si $d > \delta$} \\
\end{array}
\right.
$$
En particulier, l'élément $\overline \nabla$ est une 
$\bfA$-base de $\vH^0_{\uX}(\bfB)_\delta$.
\end{enumerate}
\end{prop}

\label{NOTA02-H0CechB}%
%
%

\subsubsection*{Avec le module $\bfB'$}

Quant au module $\bfB' = \bfA[\uX]/\langle \uP, \nabla \rangle$, l'analogue du
point ii) se reformule de façon relativement cryptique: $\forall\, d
\geqslant \delta$, $\vH^0_{\uX}(\bfB')_d = 0$ ; ou bien, de manière plus
terre à terre
$$
\forall\, d \geqslant \delta, \qquad 
\bfB'_d  \ = \  \big(\bfA[\uX]/\uPsat\big)_d
$$
Au lecteur de s'en assurer en constatant que~\ref{MiniWiebe}
s'énonce aussi sous la forme: $\langle\uP,\,\nabla\rangle^\sat_d =
\langle\uP,\,\nabla\rangle_d$ pour tout $d \geqslant \delta$.

\medskip

Quel est l'intérêt de cette formulation en cohomologie de \Cech{} dans
notre travail?  Hormis la notation~$\vH^0_{\uX}(\bfB)$ qui peut se
révéler commode (nous l'utiliserons librement dans le
chapitre~\ref{ChapSylvesterHybride}), le complexe de \Cech{}
$\vC^\sbullet_\uX(\bfA[\uX])$ de la suite régulière $\uX$ permet la mise en
place du bicomplexe défini par $C_{p,q} = \rmK^{n-p}(\uP;\bfA[\uX])
\otimes \vC^{n-q}_\uX(\bfA[\uX])$. L'étude de ce bicomplexe conduit à
l'élaboration d'un morphisme gradué~$\tau$ nommé \emph {transgression}.
\index{transgression}%
Nous n'étudierons pas ici la cohomologie de \Cech{}; nous évoquerons cependant
la transgression en section~\ref{HCech0B}.  Une retombée concrète réside dans
le fait que $\tau_0$, composante homogène de~$\tau$ de degré~$0$,
établit une relation entre le $\bfA$-module $(\bfB_\delta)^\star = \Ker\transpose
{\Syl_\delta}$ et l'idéal d'élimination~$\ElimIdeal$:
\begin{quote}
\it
Pour une suite $\uP$ quelconque, toute forme linéaire 
$\mu : \bfA[\uX]_\delta \to \bf A$ nulle sur $\uPdelta$
a la propriété suivante: $\mu(\nabla)F - \mu(F)\nabla
\in \uPdelta$ pour tout $F\in\bfA[\uX]_\delta$.
En conséquence, $\mu(\nabla) \in \ElimIdeal$.
On dispose ainsi d'une application $\bfA$-linéaire notée $\tau_0$:
$$
\tau_0 :
\begin{array}{ccl}
(\bfB_\delta)^\star = \Ker\transpose{\Syl_\delta}  & \longrightarrow & \ElimIdeal
\\  
\mu & \longmapsto & \mu(\nabla)
\end{array}
$$
Lorsque $\uP$ est régulière, cette application
$\tau_0: \mu \mapsto \mu(\nabla)$, où $\mu : \bfA[\uX]_\delta \to \bf A$
est une forme linéaire nulle sur $\uPdelta$, établit
un isomorphisme $(\bfB_\delta)^\star = \Ker\transpose{\Syl_\delta} \simeq \ElimIdeal$.
\end{quote}

\label{NOTA02-tau0}%

Nous obtiendrons le second résultat en \ref{4pointsPsuperreguliere} à
l'aide d'une preuve autonome nécessitant cependant une hypothèse
additionnelle de profondeur $\geqslant 2$ sur $\uP$ (vérifiée en
terrain générique).

\subsection{Idéaux et supplémentaires monomiaux excédentaires attachés à un format~$D$}
\label{DefIdealModuleExcedentaire}

Les définitions qui suivent sont relatives à un format de degrés $D =
(d_1, \dots, d_n)$ et indépendantes de tout système $\uP$ de
ce format. On introduit l'ensemble des indices de divisibilité d'un
monôme $X^\alpha$ (sous-entendu relativement au format $D$ que l'on ne
fait pas figurer dans la notation):
$$
\DivSeq(X^\alpha) = \big\{j \in \{1..n\} \mid X_j^{d_j} \text{ divise } X^\alpha\big\}
$$

\index{ensemble de divisibilité (d'un monôme)}%
\label{NOTA02-Div}%

\begin {defn}  \leavevmode

\noindent
Les idéaux monomiaux $1$-excédentaire et $2$-excédentaire $\Jex_2
\subset \Jex_1$ de $\bfA[\uX]$ sont définis de la manière suivante
$$
\Jex_1 \ = \ \langle X_1^{d_1}, \dots, X_n^{d_n} \rangle,
\qquad\qquad
\Jex_2 \ = \ \langle X_i^{d_i}X_j^{d_j} \mid i \neq j \rangle
$$
Le premier est le $\bfA$-module de base les monômes
divisibles par au moins un $X_i^{d_i}$ et le second celui de
base les monômes divisibles par au moins deux
$X_i^{d_i}$, $X_j^{d_j}$. 

\noindent
On note $\Jex_{1,d}$ et $\Jex_{2,d}$ leur composante homogène de degré $d$.
\end {defn}  

\index{idéal!monomial excédentaire!$\Jex_1$, $\Jex_2$}%
\label{NOTA02-Jex}%

\medskip
Grâce à la $\bfA$-base monomiale de $\bfA[\uX]$, tout idéal monomial de $\bfA[\uX]$
possède un supplémentaire canonique monomial, à savoir le $\bfA$-module de base
les monômes n'appartenant pas à l'idéal.

\index{supplémentaire!canonique monomial}%
%
%

\begin {defn} \leavevmode

\noindent  
On désigne par $\Jex_{1\setminus2}$ le supplémentaire monomial de $\Jex_2$ dans $\Jex_1$ de
sorte que $\Jex_1 = \Jex_2 \oplus \Jex_{1\setminus2}$. Il admet comme base
les monômes divisibles par un seul $X_i^{d_i}$.

\noindent
On a $\Jex_{1\setminus2} = \bigoplus_{i=1}^n \Jex_{1\setminus2}^{(i)}$ où
$\Jex_{1\setminus2}^{(i)}$ est le $\bfA$-module de base
les $X^\alpha$ vérifiant $\DivSeq(X^\alpha) = \{i\}$ \idest
\begin{center}
$\alpha_i \geqslant d_i$ \qquad \text{ et } \qquad  $\alpha_j < d_j$ pour tout $j \neq i$ 
\end{center}

\noindent
Les composantes homogènes de degré $d$ sont désignées par $\Jex_{1 \setminus 2, d}$
et $\Jid$.
\end {defn}  

\index{supplémentaire!monomial excédentaire}%
\label{NOTA02-Jex12}%
\label{NOTA02-Jex12i}%
%
%

\begin{prop} \label{LemmeDenombrementJ1moins2id}
\leavevmode
Soit $1 \leqslant i \leqslant n$. Pour $d \geqslant \delta$, 
le nombre de monômes de $\Jid$ est 
$$
\left\{
\begin{array}{ll}
\widehat d_i := \prod\limits_{j \neq i} d_j & 
\text{pour $d \geqslant \delta + 1$} \\ [0.5cm]
\widehat d_i -1 & \text{pour $d = \delta$}
\end{array}
\right.
$$
\end{prop}

\begin {proof} 
Supposons d'abord $d \geqslant \delta+1$.
L'application suivante est bijective :
$$
\begin{array}{ccl}
\big\{\text{monômes de } \Jid\big\}
         & \longrightarrow & \prod\limits_{j \neq i} \llbracket 0,\, d_j -1 \rrbracket
\\
X^\alpha & \longmapsto & (\alpha_1, \dots, \alpha_{i-1}, \alpha_{i+1}, \dots, \alpha_n)
\end{array}
$$
de bijection réciproque 
$$
\begin{array}{rcl}
  \prod\limits_{j \neq i} \llbracket 0,\, d_j -1 \rrbracket & \longrightarrow &
     \big\{\text{monômes de } \Jid\big\}  
\\
(\alpha_1, \dots, \alpha_{i-1}, \alpha_{i+1}, \dots, \alpha_n) & \longmapsto & X^\alpha 
\text{ où $\alpha_i = d-\sum\limits_{j \neq i} \alpha_j$}
\end{array}
$$
Cette dernière application est bien définie. En effet, soit 
$(\alpha_1, \dots, \alpha_{i-1}, \alpha_{i+1}, \dots, \alpha_n)$ dans l'ensemble de 
gauche (de sorte que $\alpha_j \leqslant d_j-1$) ; 
il s'agit de montrer que $\alpha_i =  d-\sum_{j \neq i} \alpha_j$ vérifie $\alpha_i \geqslant d_i$.
Pour cela, il suffit de contempler les deux extrémités de :
$$
\underbrace{\delta + 1}_{d_i + \sum\limits_{j\neq i} (d_j-1)}
\ \leqslant \ 
\underbrace{d = |\alpha|}_{\alpha_i +  \sum\limits_{j\neq i} \alpha_j}
\ \leqslant \ 
\alpha_i +  \sum\limits_{j\neq i} (d_j-1)
$$
On en déduit  que le cardinal de la base monomiale de $\Jid$ est $\widehat d_i$.

\medskip
\noindent
Pour $d = \delta$, il faut tenir compte du \MoutonNoir{} qui n'est pas dans
${\Jex_{1\setminus 2,\delta}^{(i)}}$ et supprimer du produit
$\prod\limits_{j \neq i} \llbracket 0,\, d_j -1 \rrbracket$ le $(n-1)$-uplet
$(d_1-1, \dots, d_{i-1}-1,d_{i+1}-1, \dots, d_n-1)$.
\end {proof}

\medskip

Nous verrons ultérieurement (cf. section~\ref{sousSectionMacaulayRecurrence})
une formule, dite de Macaulay en degré $d \geqslant \delta+1$, qui exprime
le résultant $\Res(\uP)$ comme quotient de deux déterminants.
Celui figurant au numérateur est indexé par les
monômes de $\Jex_{1,d}$ et celui au dénominateur par les monômes de
$\Jex_{2,d}$. Ceci explique pourquoi nous nous intéressons aux idéaux
monomiaux $\Jex_h$. En particulier, il n'est pas
inutile de connaître les quelques rares cas où le dénominateur est 1,
\idest{} les cas où $\Jex_{2,d}$ est réduit à~$0$.

\begin{prop}[Nullité de $\Jex_{2,d}$]
\label{J2deltaNul}

\leavevmode
\begin{enumerate}[\rm i)]
\item 
On a $\Jex_{2,\delta} = 0$ si et seulement si on est dans l'un des trois cas suivants :

\noindent
$(n= 2)$
ou $(n = 3$ et $d_1, d_2, d_3 \leqslant 2)$
ou ($n$ quelconque et les $d_i = 1$ sauf éventuellement un $d_j = 2$).

\item 
Pour $n = 2$, on dispose de l'équivalence $\Jex_{2,d} = 0 \Leftrightarrow d \leqslant \delta + 1$.

\item 
On a $\Jex_{2,\delta+1} = 0$ 
si et seulement si $n=2$ ou ($n \geqslant 3$ et $d_i = 1$ pour tout $i$).

\end{enumerate}
\end{prop}

\begin{proof}

Avant de commencer la preuve, remarquons que l'on a :
$$
\Jex_{2,d} = 0
\qquad \iff \qquad 
d < \min_{i\neq j}(d_i + d_j)
$$
En effet, s'il existe $i \ne j$ tels que $d \geqslant d_i+d_j$, alors 
le monôme $X_i^{d_i}X_j^{d_j}X^\beta$, où $\beta$ 
vérifie $d_i + d_j + |\beta| = d$, est dans $\Jex_{2,d}$.
Et réciproquement.

\medskip
i) 
Sans perte de généralités, on peut supposer $d_1 \leqslant d_2 \leqslant \cdots \leqslant d_n$.
On a alors, d'après la remarque initiale :
$$
\Jex_{2,\delta} = 0 
\quad \Leftrightarrow \quad 
(d_1-1) + (d_2-1) + \cdots + (d_n-1) < d_1 + d_2
\quad \Leftrightarrow \quad 
(d_3-1) + \cdots + (d_n-1) < 2
$$
Quelles sont les possibilités pour que $s = (d_3-1) + \cdots + (d_n-1)$ soit égal à $0$ ou $1$ ?
Elles sont au nombre de trois :
$n=2$ ou ($n \geqslant 3$ et $s=0$) ou ($n \geqslant 3$ et $s=1$).
Rien à discuter en ce qui concerne $n=2$.

\begin{itemize}
\item $n \geqslant 3$ et $s=0$ 
alors $d_i = 1$ pour tout $i \geqslant 3$, donc pour tout $i$, 
puisque $d_1 \leqslant d_2 \leqslant d_3$.

\item $n \geqslant 3$ et $s=1$ alors 
$d_i = 1$ pour $3 \leqslant i \leqslant n-1$ et $d_n=2$.
Si $n \geqslant 4$, cela donne $d_n = 2$ et les autres $d_i=1$;
si $n=3$, cela donne $d_3=2$ et $d_1, d_2 \leqslant 2$.
\end{itemize}

\smallskip
ii) 
Comme $n=2$, on a $d_1 + d_2 = \delta + 2$ et 
la remarque initiale s'écrit :
$\Jex_{2,d} = 0
 \iff 
d < d_1 + d_2
$, 
d'où l'équivalence annoncée.

\smallskip
iii) 
Sans perte de généralités, on peut supposer $d_1 \leqslant d_2 \leqslant \cdots \leqslant d_n$.
On a alors, d'après la remarque initiale :
$$
\Jex_{2,\delta+1} = 0 
\quad \iff \quad 
d_1 + \cdots + d_n - n+1 \leqslant d_1 +d_2-1
$$
La condition de droite est toujours vraie pour $n=2$.
Pour $n\geqslant 3$, elle s'écrit $d_3+\cdots + d_n \leqslant n-2$ : 
il~y~a~$n-2$ entiers $\geqslant 1$, donc cela équivaut à $d_i = 1$ pour tout $i \geqslant 3$, 
c'est-à-dire $d_i = 1$ pour tout $i$.
\end{proof}

\subsection{Le complexe $\rmK_{\sbullet,d}(\protect\uP)$ 
            composante homogène de degré $d$ de $\rmK_\sbullet(\protect\uP)$} 

L'algèbre extérieure $\bigwedge\big(\bfA[\uX]^n\big)$ est un
$\bfA[\uX]$-module libre de dimension $2^n$.  En notant $(e_i)_{1
  \leqslant i \leqslant n}$ la base canonique de $\bfA[\uX]^n$, il
possède une base naturelle $(e_I)_I$ indexée par les $2^n$ parties $I$
de $\{1..n\}$ où pour $I = (i_1 < \dots < i_k)$, on note
$e_I = e_{i_1} \wedge \cdots \wedge e_{i_k}$.
Les termes du complexe de Koszul de $\uP$ ne dépendent que de $n$~:
ce sont les $\rmK_k =
\bigwedge^k\big(\bfA[\uX]^n\big)$ pour $0 \leqslant k \leqslant n$
et l'on a : 
$$
\bigwedge\big(\bfA[\uX]^n\big) \ = \ \bigoplus_{k=0}^n \rmK_k,
\qquad\qquad  
\rmK_{k} \ = \ \bigoplus_{\#I = k} \bfA[\uX] e_I \ \simeq \ \bfA[\uX]^{\binom{n}{k}}
$$
On a bien entendu $\rmK_0 = \bfA[\uX]$ et $\rmK_n = \bfA[\uX]\,
e_1\wedge\cdots\wedge e_n$. Pour $k < 0$ ou
$k > n$, on pose $\rmK_k=0$.

\index{complexe composante homogène $\rmK_{\sbullet,d}(\uP)$}%
\index{degré d'homogénéité, $\bbN$-graduation de $\rmK_\sbullet(\uP)$}%
\index{graduation de $\rmK_\sbullet(\uP)$}%
\label{NOTA02-rmKP}%

\medskip
On fixe désormais un format de degrés $D = (d_1, \dots, d_n)$. Il
s'agit de graduer les termes $\rmK_k$ de manière à ce que, pour
n'importe quel système $\uP = (P_1, \dots, P_n)$ de format $D$, la
différentielle de Kozsul de $\uP$
$$
\partial_k(\uP) : \ \rmK_k \longrightarrow \rmK_{k-1} \quad
\hbox {soit homogène de degré $0$ pour n'importe quel $k$}
$$
\label{NOTA02-partialkP}%
étant entendu que l'on attribue comme poids à $P_i$ son degré $d_i$. 
Rappelons que $\partial_k = \partial_k(\uP)$ est définie par :
$$
\partial_k : \ 
e_I \ \longmapsto \ 
\sum_{i \in I} (-1)^{\varepsilon_i(I)} P_i\, e_{I \setminus \{i\}} \qquad
\begin {array}{l}
\hbox {où $\varepsilon_i(I)$ est le nombre d'éléments de $I$} \\
\hbox {strictement plus petits que $i$} \\
\end{array}    
$$
On est donc conduit à prendre $\deg (e_I) = d_I := \sum_{i \in I} d_i$ donc
$\deg(X^\alpha e_I) = |\alpha| + d_I$.
De cette manière, $\partial_k$ transforme un élément homogène de $\rmK_k$
en un élément de $\rmK_{k-1}$ homogène et de même degré.  Ainsi, en ce
qui concerne ses différentielles, le complexe de Koszul du jeu étalon
$\uX^D = (X_1^{d_1}, \dots, X_n^{d_n})$ est convenablement
$\bfA[\uX]$-gradué et il en est de même pour tout système $\uP$ de
format $D$.

\index{différentielle (de Koszul) descendante de $\uP$}%
\label{NOTA02-eI}%
\label{NOTA02-dI}%

\medskip

On peut alors considérer la composante homogène de degré $d$ de
l'algèbre extérieure ; c'est un $\bfA$-module libre possédant une base
monomiale canonique constituée des monômes \emph {extérieurs}
$X^\alpha e_I$ tels que $|\alpha| + d_I = d$. On note~$\rmK_{\sbullet,
  d}$ cette composante homogène, qui ne dépend que de $D$ (et pas de
$\uP$) ; ainsi, on~a~:
$$
\rmK_{k,d} \ = \ \bigoplus_{\#I = k} \bfA[\uX]_{d-d_I} e_I
\ = \ 
\bigoplus_{\substack {\#I = k \\ |\alpha| = d-d_I}} \bfA X^\alpha e_I
$$
\label{NOTA02-rmKdP}%
En revanche, les différentielles dépendent bien sûr de $\uP$. Lorsque
l'on voudra insister sur cet aspect, on notera $\rmK_\sbullet(\uP)$ au
lieu de $\rmK_\sbullet$. 

\index{monôme extérieur}%
%
%

\begin{theo} \label{ExactitudeKd}
Soit $d$ un entier.
Lorsque la suite $\uP$ est régulière, 
le complexe de $\bfA$-modules libres 
$\rmK_{\sbullet, d}(\uP)$ est exact.
\end{theo}

\begin{proof}
Le complexe de Koszul $\rmK_\sbullet(\uP)$ est exact 
car $\uP$ est régulière (cf.~\ref{ReguliereImpliqueCompletementSecant}).
Ainsi, pour tout $k$, le noyau de $\partial_k$ est égal à l'image de 
$\partial_{k+1}$.
En particulier, leur composante homogène de degré $d$ sont égales.
\end{proof}

\bigskip
En ce qui concerne le coefficient binomial $\binom{a}{m}$, nous
adoptons la convention de~\cite[Chap. 5, p.~154]{GKP} ; l'indice bas $m$ doit
être entier positif ou négatif tandis que l'indice haut $a$
est élément d'une $\bbQ$-algèbre commutative quelconque:
$$
\binom{a}{m} = 
\left\{
\begin{array}{ll}
\dfrac {a(a-1) \cdots (a-m+1)}{m!} & \text{si $m \geqslant 0$} \\
0 & \text{sinon} \\
\end{array}
\right.
$$
Ceci permet d'écrire pour tout $d \in \bbZ$ positif ou négatif:
$$
\dim \bfA[X_1, \dots, X_n]_d = \binom {d+n-1}{d}
$$
L'utilisation de $\binom {d+n-1}{n-1}$ au lieu de
$\binom {d+n-1}{d}$ n'est correcte que si $d >  -n$.
La plupart du temps, cette distinction est sans importance
car $d \geqslant 0$. Mais il n'en est pas de même dans la
proposition qui suit, proposition très utile pour
la gestion d'exemples.

\index{coefficient binomial}%

Pour $k = 1$ par exemple, qu'en est-il de la formule de gauche, correcte, par rapport à
celle de droite (qui ne l'est pas toujours) ?
$$
\sum_{i=1}^n \binom {n + d - d_i - 1}{d-d_i}  \qquad \text{versus} \qquad
\sum_{i=1}^n \binom {n + d - d_i - 1}{n-1}
$$
Elles coïncident si $d - d_i > -n$ pour tout $i$. 
En revanche, prenons par exemple, $n=3$, $d = 3$ et $D = (3,4,5)$ ; 
alors celle de gauche fournit 0 tandis que celle de droite 
donne le résultat incorrect $0 + 0 + 1 = 1$.

\begin{prop}\label{dminKk} \leavevmode

\begin{enumerate}[\rm i)]
\item La dimension du $\bfA$-module libre $\rmK_{k,d}$ est donnée par
$$
\dim \rmK_{k,d} = \sum_{\#I=k} \binom {n + d - d_I - 1}{d - d_I}
\qquad \text{où}\quad  d_I = \sum_{i \in I} d_i  
$$

\item 
Posons $d_{\min}(\rmK_k) = \min\limits_{\#I = k} d_I$ en convenant que $d_{\min}(\rmK_k) = +\infty$
si $k < 0$ ou $k > n$ (cas pour lesquels il n'y a pas de partie de cardinal $k$ de $\{1..n\}$). Alors:
$$
\rmK_{k,d} = 0 
\quad \iff \quad
d < d_{\min}(\rmK_k)
$$

\item
On a $d_{\min}(\rmK_k) \leqslant  d_{\min}(\rmK_{k+1})$.
En conséquence, si $\rmK_{k,d} = 0$ alors $\rmK_{k+1,d} = 0$.

\item
En degré $\delta$, le dernier terme du complexe $\rmK_{\sbullet,\delta}$ est nul, 
\idest{} 
$\rmK_{n,\delta} = 0$. 
\end{enumerate}
\end{prop}

\begin{proof} \leavevmode

\noindent i)
Découle du fait que $\rmK_{k,d}$ est la somme directe des
$\bfA[\uX]_{d-d_I}$  sur les parties $I$ de cardinal $k$.

\medskip
\noindent ii)  
La famille $(X^\alpha e_I)$ où $I$ parcourt les parties de $\{1..n\}$ de cardinal 
$k$ et $\alpha$ les $n$-uplets vérifiant $|\alpha| = d-d_I$
est une $\bfA$-base de $\rmK_{k,d}$.
Ainsi, $\rmK_{k,d} = 0$ si et seulement si pour tout $I$, on a $d - d_I < 0$.
On trouve bien l'équivalence annoncée.

\medskip
\noindent iii)
Soit $J$ une partie de cardinal $k+1$ de $\{1..n\}$ et $I \subset J$
quelconque de cardinal $k$. Alors $d_I \leqslant  d_J$, a
fortiori $d_{\min}(\rmK_k) \leqslant  d_J$. Cette dernière inégalité ayant
lieu pour toute partie $J$ de cardinal~$k+1$, on en déduit
$d_{\min}(\rmK_k) \leqslant  d_{\min}(\rmK_{k+1})$.

\smallskip
Le ``En conséquence'' s'en déduit aussitôt mais on peut également remarquer
que la présence d'un monôme $X^\alpha e_I \in \rmK_{k+1,d}$
entraîne $X^\alpha X_i^{d_i} e_{I \setminus i} \in \rmK_{k,d}$ 
pour n'importe quel $i$ de la partie non vide~$I$.

\medskip
\noindent iv) 
Utilisons le point i) :
la dimension de $\rmK_{n, \delta}$ est un coefficient binomial 
(la somme ne porte que sur $I = \{1,\dots,n\}$)
dont l'indice bas 
est $\delta - (d_1+\cdots+d_n)$. Cet indice est strictement négatif 
par définition du degré critique $\delta$, et par conséquent, 
le coefficient binomial est nul.
\end{proof}

\bigskip
Pour $k=1$ et $k=0$, la formule donnée en i) permet de
donner le format de la matrice de 
l'application de Sylvester 
$\Syl_\delta : \rmK_{1,\delta} \to \rmK_{0,\delta} = \bfA[\uX]_\delta$
en degré $\delta$, qui va jouer un rôle important dans la suite.
En général, cette matrice est \og plus large que haute\fg{} \idest{}
$\dim \rmK_{1,\delta} > \dim \rmK_{0,\delta}$. Voici quelques \emph{exceptions}:
$$
\begin {array}{c|c|c|c|c|c|c|c|c|c|c|c|c|c|c}
D                   &(1,1,2)  &(1,1,3) &(1,2,2) &(1,2,3) &(2,2,2) &(2,2,3)\\[1mm]
\hline
\vrule height12pt depth3pt width0pt
\dim \rmK_{1,\delta} &2        &6       &5       &10      &9       &15\\[1mm]
\hline
\vrule height12pt depth3pt width0pt
\dim \rmK_{0,\delta} &3        &6       &6       &10      &10      &15\\
\end {array}
$$


\subsubsection*{Caractéristique d'Euler-Poincaré $\chi_d = \chi_d(D)$ de $\rmK_{\sbullet,d}(X^D)$}

\index{caractéristique d'Euler-Poincaré}%
\label{NOTA02-chid}%

La caractéristique d'Euler-Poincaré $\chi_d$ de $\rmK_{\sbullet,d}(\uP)$ 
est par définition l'entier égal à la somme alternée des dimensions des termes de ce complexe,
en commençant par le terme de degré (homologique) 0:
$$
\chi_d \ =\ \dim \rmK_{0,d} - \dim \rmK_{1,d} + \cdots 
\qquad \ =\ 
\sum_{k=0}^n (-1)^k \dim \rmK_{k,d}
$$
Comme $\rmK_{k,d}$ ne dépend que du format $D$ (et pas de la suite $\uP$), 
il en est de même de $\chi_d$.
Dans la suite, on rappellera parfois la dépendance en $D$ via la notation $\chi_d(D)$.

\begin{prop}[Caractéristique d'Euler-Poincaré attachée à $D$ en degré $d$] 
\label{chiPropriete}
\leavevmode
\begin{enumerate}[\rm i)]
\item 
On a 
$\chi_d = \dim \big(\bfA[\uX]/ \langle X_1^{d_1}, \dots, X_n^{d_n} \rangle \big)_d
= \# \big\{ \alpha \in \bbN^n \mid |\alpha | = d \text{ et } \alpha \preccurlyeq \emouton \big\}$.

\item 
L'entier $\chi_d$ est le coefficient de degré $d$ du \emph{polynôme}
$\displaystyle\dfrac{(1-t^{d_1}) \cdots (1-t^{d_n})}{(1-t)^n}$, unitaire de degré $d$.

En particulier, 
$$
\chi_d
\ = \ 
\left\{
\begin{array}{ll}
0 & \text{pour $d \geqslant \delta+1$} \\
[0.2cm]
1 & \text{pour $d = \delta$} 
\end{array}
\right.
$$
\item
On dispose de l'égalité:
$$
\sum_{d=0}^\delta \chi_d  = d_1d_2 \cdots d_n
$$  
\end{enumerate}
\end{prop}

\begin{proof}
i) 
Pour le jeu étalon $\uX^D$, les $X^\alpha$ avec $|\alpha | = d$ et
$\alpha \preccurlyeq \emouton$ forment une base du module
libre $\bfB_d = \bfA[\uX]_d/\langle\uX^D\rangle_d$, d'où la deuxième égalité.

Pour montrer la première égalité, 
prenons la suite régulière $\uP = \uX^D$. Le complexe 
de modules libres $\rmK_{\sbullet,d}(\uX^D)$ 
est exact (cf.~\ref{ExactitudeKd}), et il résout le module \textit{libre} $\bfB_d$.  
Comme la dimension des termes de cette résolution 
est indépendante de l'anneau $\bfA$, 
on peut prendre comme anneau celui qui nous arrange: par exemple un corps ou
bien $\bfA = \bbZ$. Dans ce cas, pour une application linéaire $u : E
\to F$ entre deux modules libres, $\Im u$ et $\Ker u$ sont libres et
$\dim E = \dim \Ker u + \dim \Im u$.  En appliquant cette égalité
à la dimension de chaque terme intervenant dans la somme alternée~$\chi_d$
et en utilisant l'exactitude du complexe, on obtient $\chi_d = \dim \bfB_d$.

\smallskip
ii) Le nombre de $\alpha = (\alpha_1,\dots,\alpha_n) \preccurlyeq\emouton$ 
tels que $\sum_{i=1}^n \alpha_i = d$ est le coefficient de degré $d$ du polynôme 
$$
\sum_{0 \leqslant \alpha_1 < d_1} t^{\alpha_1}
\times 
\cdots 
\times 
\sum_{0 \leqslant \alpha_n < d_n} t^{\alpha_n}
\ = \ 
\dfrac{1-t^{d_1}}{1-t} \times \cdots \times \dfrac{1-t^{d_n}}{1-t}
$$
Ce polynôme est unitaire de degré $\delta$, donc $\chi_\delta = 1$ 
et $\chi_d = 0$ pour $d \geqslant \delta+1$.

\smallskip
iii) D'après le point précédent:
$$
\sum_{d=0}^\delta \chi_d\,t^d = \prod_{i=1}^n \dfrac {1 - t^{d_i}} {1-t}
$$
Il suffit d'y faire $t := 1$ en ayant remarqué que $\big((1-t^{d_i})/(1-t)\big)_{t := 1} = d_i$.

On peut également utiliser
$$
\sum_{d=0}^\delta \chi_d = \dim \bfA[\uX]/ \langle X_1^{d_1}, \dots, X_n^{d_n} \rangle
$$
et le fait que la famille $(x^\alpha)_{0 \le \alpha_i < d_i}$, de
cardinal $d_1 d_2\cdots d_n$, est une $\bfA$-base du quotient de
droite.
\end{proof}

\medskip

La preuve fait apparaître la famille $X^\alpha$ avec $|\alpha | = d$
et $\alpha \preccurlyeq \emouton$, comme base monomiale d'un
supplémentaire de $\langle\uX^D\rangle_d$ dans $\bfA[\uX]_d$, et donc
également comme base de $\bfB_d$.
Nous notons $\Jex_{0\setminus1}$ le supplémentaire monomial de
$\Jex_1 \buildrel {\rm def} \over = \langle\uX^D\rangle$ dans
$\Jex_0 \buildrel {\rm def} \over = \bfA[\uX]$. 
La notation $\Jex_0$ se veut uniforme avec la notation 
$\Jex_h$ qui sera introduite ultérieurement pour tout $h \geqslant 0$
pour désigner une collection d'idéaux monomiaux attachée à $D$. 
En résumé, pour la caractéristique d'Euler-Poincaré 
$\chi_d$ de $\rmK_{\sbullet,d}(\uX^D)$:
$$
\chi_d = \dim \big(\bfA[\uX]_d / \Jex_{1,d} \big) =
\dim \Jex_{0\setminus1,d} =
\# \big\{ \alpha \in \bbN^n \mid 
|\alpha| = d \text{ et } \alpha \preccurlyeq \emouton \big\}
$$

Nous regroupons ci-dessous un certain nombre d'informations concernant
la suite $(\chi_d)_{d\le 0\le \delta}$ qui vérifie des propriétés
intervenant dans plusieurs domaines: combinatoire, algèbre commutative
combinatoire, topologie simpliciale, théorie des nombres, théorie des
probabilités, etc. Parmi ces propriétés, il y a celle assez banale de
symétrie: une suite $u = (u_0, \dots, u_p)$ est dite symétrique si
$u_{p-i} = u_i$ pour $0 \le i \le p$.  Une autre propriété peut-être
moins connue est l'unimodalité: $u$ est dite unimodale si elle est
croissante puis décroissante i.e.  s'il existe un indice $k_0$ tel que
$u_k \le u_{k+1}$ pour $k < k_0$ et $u_k \ge u_{k+1}$ pour $k \ge
k_0$. Uniquement pour les définitions en langue française, la lectrice
pourra consulter l'article de Michel Balazard \cite{Balazard}; le
domaine de la théorie des nombres ne nous concerne pas ici mais il est
bon de remarquer dans la bibliographie de cet article les références à
Richard Stanley, champion (entre autres) de la combinatoire algébrique
et topologie simpliciale.

\smallskip

La vérification en est laissée au lecteur, à l'exception du dernier point
qui nécessite un peu de travail.

\label{NOTA02-Jex01}%
%
%

\begin{prop}[Propriétés de $\chi_d$] 
\label{chiProprieteComplements} \leavevmode
Dans le contexte d'un format $D$ de degrés:

\begin{enumerate}[\rm i)]
\item
{\bf Symétrie}.  
Pour $d + d' = \delta$, on dispose de la bijection suivante ({\rm mouton-swap})
$$
\begin{array}[t]{rcl}
\Jex_{0\setminus1,d} & \longrightarrow & \Jex_{0\setminus1,d'} \\ [0.1cm]
X^\alpha & \longmapsto & X^{\emouton-\alpha}
\end{array}
$$
On a donc $\chi_d = \chi_{d'}$, entier que l'on note parfois $\chi_{d,d'}$.

\item 
\begin{itemize}
\item 
Pour $0 \leqslant d \leqslant \delta$, 
on a une inclusion stricte $\Jex_{1,d} \subsetneq \bfA[\uX]_d$ de sorte que $\chi_d > 0$. 

\smallskip

De plus $\chi_0 = \chi_\delta = 1$: cela traduit le fait que $1$ (resp. $X^\emouton$)
est le seul monôme de degré $0$ (resp.~$\delta$) non divisible par un $X_i^{d_i}$.

\smallskip

\item 
Pour $d \geqslant \delta+1$, on a $\Jex_{1,d} = \bfA[\uX]_d$ de sorte que 
$\chi_d = 0$.
\end{itemize}

\item 
En notant $s_d$ le rang de l'application de Sylvester de $\uX^D$ en degré $d$,
on dispose des égalités 
$$
s_d = \dim \Jex_{1,d}
\qquad \text{et} \qquad 
\dim \bfA[\uX]_d \, = \, s_d + \chi_d
$$

\item
{\bf Unimodalité}.
La fonction $d \mapsto \chi_d$ est croissante sur 
$\big\llbracket 0, \lfloor \frac{\delta}{2} \rfloor \big\rrbracket$ 
et décroissante sur 
$\big\llbracket \lfloor \frac{\delta}{2} \rfloor, \delta \big\rrbracket$.
Elle admet donc un maximum en $\lfloor \frac{\delta}{2} \rfloor$
(si $\delta$ est impair, ce maximum est atteint en 
$\lfloor \frac{\delta}{2} \rfloor$ et en $\lfloor \frac{\delta}{2} \rfloor + 1$).
\end{enumerate}
\end{prop}

\label{NOTA02-sd}%
%
%

\begin{proof}

iv) D'après~\ref{chiPropriete}, on a 
$$
\displaystyle \sum\limits_{d=0}^\delta \chi_d \,t^d 
\ = \ 
\sum_{0 \leqslant \alpha_1 < d_1} t^{\alpha_1}
\times 
\cdots 
\times 
\sum_{0 \leqslant \alpha_n < d_n} t^{\alpha_n}
$$
Cette égalité traduit le fait que la suite $\chi = (\chi_0, \chi_1,
\dots, \chi_\delta)$ est le produit de convolution des $n$ suites $(1,
\dots, 1)$ de longueurs respectives $d_1$, $d_2$, \dots, $d_n$.  On
invoque alors le lemme~\ref{SuitePointue} qui suit.
\end{proof}

Convenons de dire\footnote{%
Autrement dit, une suite symétrique et unimodale.
Dans une première rédaction, dont nous voulons garder des traces pour souvenir, nous disions:
``Il s'agit d'une définition purement locale; pour les personnes
désirant disposer d'un nom au lieu de \pointu{}, nous leur proposons \og suite de type pic\fg''.
Ce discours est devenu obsolète depuis l'introduction de l'unimodalité.
}
qu'une suite $\ua = (a_0, \ldots, a_p)$ d'entiers naturels est de type \pointu{} si~:
\begin{center}
$\ua$ est symétrique (\idest{} $a_{p-i} = a_i$ pour tout $i$), \quad
la sous-suite $(a_0, \dots, a_{p'})$ est croissante où 
$p' = \lfloor \frac{p}{2} \rfloor$
\end{center}

\begin{lem} \label{SuitePointue}
\leavevmode
\begin{enumerate}[\rm i)]
\item
La somme de deux suites de type \pointu{} de {\rm même longueur} est encore de type \pointu.

\item
Soit deux suites  $\ua = (a_0, \ldots, a_p)$ et $\ub = (b_0, \ldots, b_q)$ de type \pointu{} 
non nécessairement de même longueur.

Alors la suite $\uc = (c_0, \ldots, c_{p+q}) = \ua \star \ub$, 
produit de convolution de $\ua$ et $\ub$, est encore de type \pointu.
\end{enumerate}
\end{lem}

\begin{proof}
Le point i) est facile.
Pour le deuxième point, introduisons $F = \sum_{i=0}^p a_it^i$ et $G = \sum_{j=0}^q b_jt^j$
de sorte que $\sum_{k=0}^{p+q} c_k t^k = FG$.
La seule chose non immédiate à vérifier est l'inégalité
$$
c_{k-1} \leqslant c_k  \quad \hbox {pour $2k \leqslant p+q$}
$$
Posons $q' = \lfloor \frac{q}{2} \rfloor$ et introduisons les suites 
$\ub^{(0)}, \ub^{(1)}, \ldots, \ub^{(q')}$, qui sont de type \pointu{} et de longueur~$q+1$:
$$
\ub^{(0)} = (1,1, \ldots, 1,1), \qquad
\ub^{(1)} = (0, 1,\ldots,1, 0), \qquad
\ub^{(2)} = (0,0,1 \ldots,1,0,0), \qquad
\dots
$$
On écrit :
$$
\ub \ =\ 
\beta_0\,\ub^{(0)} + \beta_1\,\ub^{(1)} + \beta_2\,\ub^{(2)} + \cdots + \beta_{q'}\,\ub^{(q')}
\qquad \hbox {avec} \quad
\beta_0 = b_0, \quad \beta_i = b_i-b_{i-1}\hbox { pour $1 \leqslant i \leqslant q'$} 
$$
On a $\beta_i \in \bbN$ pour tout $i$. L'égalité
$$
\ua\star\ub 
\ =\  
\beta_0\big(\ua\star\ub^{(0)}\big) \, + \, 
\beta_1\big(\ua\star\ub^{(1)}\big) \, + \, 
\beta_2\big(\ua\star\ub^{(2)}\big) \, + \, 
\cdots \, + \,  
\beta_{q'}\big(\ua\star\ub^{(q')}\big)
$$
et le point i) permettent de se ramener au cas où $\ub$ est une des suites
$\ub^{(j)}$. Traitons par exemple $\uc = \ua\star\ub^{(0)}$; on a donc
$\sum_{k=0}^{p+q} c_kt^k = FG$ avec $G = \sum_{j=0}^q t^j$.  On
prolonge $\ua$ par $a_i = 0$ pour $i< 0$ ou $i> p$ et l'on a encore
$a_i = a_{p-i}$ et $\ua$ croissante sur $\bbZ_{\leqslant p'}$.  En utilisant
$FG - tFG = (1-t^{q+1})F$, il vient :
$$
c_k - c_{k-1} 
\ \overset{\rm (1)}{=} \ 
a_k - a_{k-(q+1)} 
\ \overset{\rm (2)}{=} \ 
a_{p-k} - a_{k-(q+1)} 
$$
On a ou bien $2k \leqslant p$ (\idest{} $k \leqslant p'$) 
ou bien $2(p-k) \leqslant p$ (\idest{}  $p-k \leqslant p'$). 
Dans le premier cas, on conclut avec l'égalité $(1)$.
Dans le second, on conclut avec l'égalité $(2)$ en 
ayant remarqué que $p-k \geqslant k-(q+1)$ (puisque $2k \leqslant p+q$).
\end{proof}

\subsubsection*{Rangs attendus des différentielles d'un complexe de modules libres}
\label{RangAttendu}

Dans notre étude, vont intervenir plusieurs notions de rang: le rang
d'une application linéaire entre deux modules libres de dimension
finie, le rang d'un module librement résoluble, 
ou encore celui d'un module de présentation finie (lorsque ce rang
existe). 
Nous commençons à en parler de manière élémentaire ici.
Mais, en ce qui concerne les résultats non
triviaux dont nous avons besoin, c'est principalement la théorie des résolutions
libres finies et la structure multiplicative qui fourniront le
bon cadre permettant d'énoncer les résultats et de les prouver,
cf. le chapitre~\ref{ChapStructureMultiplicative}.

\medskip

Pour une application linéaire $u : E \to F$ entre deux modules libres
de dimension finie, la notion la plus simple à décrire est \og $u$ est
de rang $< s$\fg{} qui signifie que tout mineur d'ordre $s$ de $u$ est
nul.  En notant $\calD_s(u)$ \emph{l'idéal déterminantiel d'ordre $s$}
de $u$ (\idest{} l'idéal engendré par les mineurs d'ordre $s$), cela
s'écrit donc $\calD_s(u) = 0$. Ceci ne précise pas pour autant ce qu'est le rang
de $u$.

\index{idéal!déterminantiel}%
\label{NOTA02-calD}%

\bigskip
Pour comprendre la définition qui vient, prenons provisoirement
comme anneau de base un corps. Si $u_n : F_n \to F_{n-1}$
est injective, alors le rang de $u_n$ (au sens habituel
sur les corps) est $\dim F_n$. Si la suite ci-dessous est
exacte:
$$
0 \to \xymatrix @M=0.4pc{F_n \ar[r]^-{u_n} & F_{n-1} \ar[r]^-{u_{n-1}} & F_{n-2}}
$$
alors le rang de $u_{n-1}$ est $\dim F_{n-1} - \dim F_n$. Et ainsi de suite.

\medskip
Ceci conduit, pour un anneau de base quelconque, à associer
à tout  complexe libre $F_\sbullet$, \emph {non nécessairement exact}, 
de longueur $\leqslant n$
$$
0 \to \xymatrix @M=0.4pc{
F_n \ar[r]^-{u_n} & F_{n-1} \ar[r]^-{u_{n-1}} & \quad \cdots \quad \ar[r] &
F_2 \ar[r]^-{u_2} & F_1 \ar[r]^-{u_1} & F_0 
}
$$
la suite des \og rangs attendus de ses différentielles\fg.
Elle est définie de la manière suivante:
$$
r_{n+1} = 0, \qquad  r_{k+1} + r_k = \dim F_k 
\quad \text{pour $k \in \{n,\, n-1, \dots, 1,\, 0\}$}
$$
On a donc:
$$
r_k \ = \ 
\dim F_k - \dim F_{k+1} + \cdots + (-1)^{n-k} \dim F_n
$$
et on dit que $r_k$ est le rang attendu de $u_k$. Quant à $r_0$,
ce n'est autre que la caractéristique d'Euler-Poincaré du complexe $F_\sbullet$,
à voir également comme $r_0 = \dim F_0 - r_1$.

Quand le complexe est exact, ce rang attendu $r_k$ est le rang
de $u_k$ au sens suivant: l'idéal déterminantiel $\calD_{r_k+1}(u_k)$
est nul tandis que $\calD_{r_k}(u_k)$ est un idéal fidèle,
cf. le chapitre \ref{ChapStructureMultiplicative}.

\index{rang!attendu (d'une différentielle dans un complexe)}%
\label{NOTA02-rk}%
%
%


\subsubsection*{Famille des rangs attendus $r_{k,d} = r_{k,d}(D)$ attachée à $D$}

On peut ainsi attacher à $D = (d_1, \dots, d_n)$ la famille des rangs
attendus du complexe $\rmK_{\sbullet,d}(\uX^D)$, composante homogène de
degré $d$ du complexe de Koszul $\rmK_\sbullet(\uX^D)$:
$$
r_{k,d} = r_{k,d}(D) = \dim \rmK_{k,d} - \dim \rmK_{k+1,d} + \cdots + (-1)^{n-k} \dim \rmK_{n,d}
$$
Deux rangs sont particulièrement importants: $r_{0,d} = \chi_d$ et
$s_d := r_{1,d}$.  Choisi pour évoquer Sylvester, l'entier $s_d$, rang
attendu de l'application de Sylvester $\Syl_d : \rmK_{1,d} \to
\rmK_{0,d} = \bfA[\uX]_d$, est donc défini par
$$
s_d \ = \ 
\dim \bfA[\uX]_d - \chi_d
$$
Comme $\chi_\delta = 1$, on a $s_\delta = \dim \bfA[\uX]_\delta - 1$. Ne pas
oublier dans ces notations l'absence du format implicite $D = (d_1, \dots, d_n)$.

\label{NOTA02-rkd}%
\label{NOTA02-sdelta}%

\medskip
D'après le résultat général relatif aux différentielles
d'un complexe exact (paragraphe précédent), les idéaux déterminantiels 
d'ordre $r_{k,d}+1$ et $r_{k,d}$ de la  différentielle $\partial_{k,d}(\uP)$ du complexe 
$\rmK_{\sbullet, d}(\uP)$ possèdent des propriétés antinomiques
de régularité (nullité versus fidélité) dans le contexte suivant.

\begin{prop}[Idéaux déterminantiels de $\rmK_{\sbullet, d}(\uP)$]
\label{IdeauxDeterminantielsKP}

Soient $1 \leqslant k \leqslant n$ et $d$ un entier.

\begin{enumerate}[\rm i)]

\item 
Pour $\uP$ régulière, l'application linéaire $\partial_{k,d}$ est de rang $r_{k,d}$ :
$$
\calD_{r_{k,d}\,+\,1}(\partial_{k,d}) = 0 
\qquad \text{tandis que } \qquad 
\calD_{r_{k,d}}(\partial_{k,d}) \text{ est fidèle.}
$$

\item 
Pour toute suite $\uP$, l'application linéaire $\partial_{k,d}$ est de rang 
$\leqslant r_{k,d}$ :
$$
\calD_{r_{k,d}\,+\,1}(\partial_{k,d}) = 0
$$
\end{enumerate}
\end{prop}

\begin{proof}
D'après~\ref{ExactitudeKd}, le complexe de modules libres 
$\rmK_{\sbullet, d}(\uP)$ est exact. 
La théorie des résolutions libres finies permet de conclure, cf le
chapitre \ref{ChapStructureMultiplicative}.

L'égalité de gauche est vérifiée pour la suite générique qui est régulière.
Puisque cette égalité se spécialise, on a 
$\calD_{r_{k,d}+1}(\partial_{k,d}) = 0$ pour une suite $\uP$ \textit{quelconque}.
\end{proof}

\subsubsection*{Majoration du rang de l'application de Sylvester en degré $\delta$ : preuve directe}
\label{rankSyldeltaDirectProof}

En prenant $k=1$ dans le résultat précédent, 
on obtient que l'application de Sylvester $\Syl_d = \Syl_d(\uP)$ 
en degré $d$ d'une suite régulière $\uP$ vérifie 
$$
\calD_{s_d+1}(\Syl_d) = 0 
\qquad \text{tandis que } \qquad 
\calD_{s_d}(\Syl_d) \text{ est fidèle}
$$
L'égalité de gauche est même vérifiée pour une suite $\uP$ \textit{quelconque}.
On va en fournir une preuve \og directe\fg{} dans le cas
particulier $d = \delta$ \idest{} 
l'application de Sylvester $\Syl_\delta :
\rmK_{1,\delta} \to \bfA[\uX]_\delta$ est de rang $< \dim
\bfA[\uX]_\delta$. Précisons le sens de \og direct\fg{}: la preuve se
passe de la théorie des
résolutions libres mais utilise de manière essentielle le théorème de
Wiebe, ce qui montre le caractère non trivial du résultat.

\begin {prop}[Majoration générale du rang de $\Syl_\delta$] \label{rangSyldelta}
Pour toute suite $\uP = (P_1, \ldots, P_n)$, l'application de Sylvester de $\uP$
en degré $\delta$ est de rang strictement plus petit que 
$\dim \bfA[\uX]_\delta$.
\end {prop}

\begin {proof}
Il s'agit de montrer que tous les mineurs d'ordre $\dim
\bfA[\uX]_\delta$ de $\Syl_\delta$ sont nuls.  Ce résultat se
spécialise, donc on peut se permettre de supposer la suite $\uP$
générique. En particulier, $\uP$ est régulière et on peut appliquer~\ref{MiniWiebe}-i) :
$$
\forall\, a \in \bfA, \qquad 
a\nabla \in \uPdelta = \Im\Syl_\delta
\quad \implies \quad 
a = 0
$$
où $\nabla$ un déterminant bezoutien de $\uP$.
On conclut en utilisant le lemme suivant avec $v = \nabla$.

\begin{quote}
\it
Soit $u : E \to F$ une application linéaire entre deux modules
libres. On suppose qu'il existe $v \in F$ vérifiant:
$$
\forall\, a \in \bfA, \qquad 
a v  \in \Im u 
\quad \implies \quad 
a = 0
$$
Alors, $u$ est de rang $< f := \dim F$ \idest{} tous les
mineurs d'ordre $f$ de $u$ sont nuls.
\end{quote}

Soit $M \in \bbM_f(\bfA)$ une sous-matrice carrée de $u$ d'ordre $f$.
On a $\det(M)\,v \in \Im u$ puisque $\det(M)\,v = M\, w$ avec
$w = \widetilde {M}\, v$. En conséquence, $\det(M)\,v \in \Im u$ 
et l'hypothèse implique $\det(M) = 0$.
\end{proof}

\medskip
Il y a cependant quelques rares cas exceptionnels pour lesquels on
peut procéder vraiment directement pour montrer le résultat précédent.
Bien que ne reflétant absolument pas le cas général, 
illustrons un de ces cas.
On prend $D = (d_1, d_2, d_3)$ et on suppose $\Syl_\delta$ \textit{carrée}, 
de sorte qu'il s'agit de montrer que le déterminant de $\Syl_\delta$ est nul.
Comme $\rmK_{n,\delta} = 0$
(cf.~\ref{dminKk}-iv) et $\chi_\delta = 1$, on a :
$$
\dim\rmK_{3,\delta} = 0, \qquad\qquad
\dim\rmK_{2,\delta} = 1, \qquad\qquad
\dim\rmK_{1,\delta} = \dim\rmK_{0,\delta}
$$
et c'est l'application $\partial_{2,\delta} : \rmK_{2,\delta} \to
\rmK_{1,\delta}$ qui fournit un vecteur \emph {non nul}
du noyau de~$\Syl_\delta$ prouvant que $\det(\Syl_\delta) = 0$.
Par exemple, pour $D = (1,2,3)$, on a $\dim \rmK_{1,\delta} = \dim
\rmK_{0,\delta} = 10$ et pour 
{\small
$$
\setlength{\arraycolsep}{0.5\arraycolsep}
\begin {array}{ccl}
P_1 &=&
a_{1}X_{1} + a_{2}X_{2} + a_{3}X_{3}
\\
P_2 &=&
b_{1}X_{1}^{2} + b_{2}X_{1}X_{2} + b_{3}X_{1}X_{3} + b_{4}X_{2}^{2} + b_{5}X_{2}X_{3} + b_{6}X_{3}^{2}
\\
P_3 &=&
c_{1}X_{1}^{3} + c_{2}X_{1}^{2}X_{2} + c_{3}X_{1}^{2}X_{3} + c_{4}X_{1}X_{2}^{2} + c_{5}X_{1}X_{2}X_{3} + 
c_{6}X_{1}X_{3}^{2} + c_{7}X_{2}^{3} + c_{8}X_{2}^{2}X_{3} + c_{9}X_{2}X_{3}^{2} + c_{10}X_{3}^{3}
\\
\end {array}
$$
}
les différentielles de Koszul sont :
$$
\Syl_\delta = \partial_{1,\delta} \ = \ 
\NorthEastBordermatrix{
\Veti{X_{1}^{2}\,e_{1}} &\Veti{X_{1}X_{2}\,e_{1}} &\Veti{X_{1}X_{3}\,e_{1}} &\Veti{X_{2}^{2}\,e_{1}} &\Veti{X_{2}X_{3}\,e_{1}} &\Veti{X_{3}^{2}\,e_{1}} &\Veti{X_{1}\,e_{2}} &\Veti{X_{2}\,e_{2}} &\Veti{X_{3}\,e_{2}} &\Veti{e_{3}} & \\
a_{1} & . & . & . & . & . & b_{1} & . & . & c_{1} & \Heti{X_{1}^{3}} \\
a_{2} & a_{1} & . & . & . & . & b_{2} & b_{1} & . & c_{2} & \Heti{X_{1}^{2}X_{2}} \\
a_{3} & . & a_{1} & . & . & . & b_{3} & . & b_{1} & c_{3} & \Heti{X_{1}^{2}X_{3}} \\
. & a_{2} & . & a_{1} & . & . & b_{4} & b_{2} & . & c_{4} & \Heti{X_{1}X_{2}^{2}} \\
. & a_{3} & a_{2} & . & a_{1} & . & b_{5} & b_{3} & b_{2} & c_{5} & \Heti{X_{1}X_{2}X_{3}} \\
. & . & a_{3} & . & . & a_{1} & b_{6} & . & b_{3} & c_{6} & \Heti{X_{1}X_{3}^{2}} \\
. & . & . & a_{2} & . & . & . & b_{4} & . & c_{7} & \Heti{X_{2}^{3}} \\
. & . & . & a_{3} & a_{2} & . & . & b_{5} & b_{4} & c_{8} & \Heti{X_{2}^{2}X_{3}} \\
. & . & . & . & a_{3} & a_{2} & . & b_{6} & b_{5} & c_{9} & \Heti{X_{2}X_{3}^{2}} \\
. & . & . & . & . & a_{3} & . & . & b_{6} & c_{10} & \Heti{X_{3}^{3}} \\
}
\qquad\qquad
\partial_{2,\delta} \ = \ 
\NorthEastBordermatrix{
\Veti{e_{12}} & \\
-b_{1} & \Heti{X_{1}^{2}\,e_{1}} \\
-b_{2} & \Heti{X_{1}X_{2}\,e_{1}} \\
-b_{3} & \Heti{X_{1}X_{3}\,e_{1}} \\
-b_{4} & \Heti{X_{2}^{2}\,e_{1}} \\
-b_{5} & \Heti{X_{2}X_{3}\,e_{1}} \\
-b_{6} & \Heti{X_{3}^{2}\,e_{1}} \\
a_{1} & \Heti{X_{1}\,e_{2}} \\
a_{2} & \Heti{X_{2}\,e_{2}} \\
a_{3} & \Heti{X_{3}\,e_{2}} \\
. & \Heti{e_{3}} \\
}
$$
Les matrices sont positionnées de façon à pouvoir vérifier la nullité de $\partial_{1,\delta}\circ\partial_{2,\delta}$!

\begin{prop}[Régularité d'un déterminant bezoutien]
\label{BezoutienRegulier}
Tout déterminant bezoutien d'une suite régulière est un élément régulier.
\end{prop}

\begin{proof}
D'après le lemme de McCoy~\ref{McCoyPolyLemma}, il suffit de montrer que
l'idéal contenu $\rmc(\nabla)$, idéal de $\bfA$ engendré par les
coefficients de~$\nabla$, est d'annulateur réduit à~$0$.  Soit donc $a
\in \bfA$ tel que $a\nabla = 0$.  
Comme $\big(\uP : \nabla\big)_0 = 0$ (d'après~\ref{MiniWiebe}-i), 
on en déduit $a=0$.

\smallskip
\noindent  
Remarquons que nous
n'avons utilisé ici qu'une faible conséquence de $\big(\uP :
  \nabla\big)_0 = 0$ : en effet, l'égalité $a\nabla = 0$ est plus
forte que la simple appartenance $a\nabla\in\uPdelta$.
\end{proof}

\subsubsection*{Retour sur le module $\bfB_d$}

Un module est librement résoluble s'il admet une résolution libre finie.
Lorsque $\uP$ est régulière, c'est le cas du module $\bfB_d$  qui est résolu par le
complexe $\rmK_{\sbullet, d}(\uP)$ composante homogène de degré $d$ du Koszul de $\uP$.

\index{module!librement résoluble}%

\begin{prop} \label{ResolutionQuotients}
Soit $\uP$ une suite régulière.

\begin{enumerate}[\rm i)]
\item 
Pour tout $d$, le $\bfA$-module $\bfB_d$ est librement résoluble 
de rang $\chi_d = 
\left\{
\begin{array}{ll}
0 & \text{pour $d \geqslant \delta+1$} \\
[0.2cm]
1 & \text{pour $d = \delta$} 
\end{array}
\right.
$

\item 
Le $\bfA$-module $\bfB'_\delta$ est librement résoluble de rang $0$.
\end{enumerate}
\end{prop}

\begin{proof}
i) 
Puisque $\uP$ est régulière, 
le complexe de $\bfA$-modules libres $\rmK_{\sbullet,d}(\uP)$ est exact 
(cf.~\ref{ExactitudeKd}) et, par définition de $\bfB$, il résout $\bfB_d$.
Le rang de $\bfB_d$ est alors donné par la caractéristique d'Euler-Poincaré 
de ce complexe qui vaut $\chi_d$ (cf.~\ref{chiPropriete} pour les valeurs annoncées).

ii) 
Considérons la composante homogène en degré $\delta$ du complexe de Koszul,
modifié en degré homologique $1$ par 
$\widetilde {\rmK_{1,\delta}} = \rmK_{1,\delta} \oplus \bfA$ 
et 
$\widetilde{\partial_{1,\delta}} = \partial_{1,\delta} \oplus \rm{mult}_\nabla$:
$$
\widetilde {\partial_{1,\delta}} :\  
\widetilde{\rmK_{1,\delta}} \ \longrightarrow \ \rmK_{0,\delta}, \quad 
\qquad
U_1e_1 + \cdots + U_ne_n \, \oplus\, u 
\ \longmapsto\ 
\sum_{i=1}^n U_i P_i + u\nabla
$$
La deuxième différentielle 
$\widetilde {\partial_{2,\delta}} : 
\rmK_{2,\delta} \longrightarrow \widetilde{\rmK_{1,\delta}}$ 
est naturellement définie par 
$\widetilde{\partial_{2,\delta}}(x) = \partial_{2,\delta}(x) \oplus 0$.
Ce nouveau complexe 
$$
0 \to
\xymatrix @M=0.4pc @C=1cm{
\rmK_{n,\delta} = 0 \ar[r]^-{\partial_{n,\delta}} &
\rmK_{n-1,\delta} \ar[r]^-{\partial_{n-1,\delta}} & \quad \cdots \quad \ar[r] &
\rmK_{2,\delta} \ar[r]^-{\widetilde{\partial_{2,\delta}}} & 
\widetilde {\rmK_{1,\delta}} 
\ar[r]^-{\widetilde{\partial_{1,\delta}}} & \rmK_{0,\delta}
}
$$
est clairement de caractéristique d'Euler-Poincaré $\chi_\delta - 1$, 
donc nulle. 
Il est exact en degré homologique $k \geqslant 2$ car la suite $\uP$ est régulière.
Examinons l'exactitude en degré homologique $1$.
Soit $(\underline U, u) \in \widetilde{\rmK_{1,\delta}}$ tel que 
$\widetilde{\partial_{1,\delta}}(\underline U \oplus u) = 0$.
Par définition, on obtient alors 
$\partial_{1,\delta}(\underline U) + u\nabla = 0$, 
d'où $u\nabla \in \uPdelta$.
Or, d'après~\ref{MiniWiebe}-i), 
on a $\big(\uP : \nabla\big)_0 = 0$, donc $u = 0$. 
Puis $\partial_{1,\delta}(\underline U) = 0$.
Par exactitude du complexe de Koszul en degré homologique $1$, on obtient 
$\underline U \in \Im \partial_{2,\delta}$. Finalement, on obtient 
$(\underline U,u) \in \Im \widetilde{\partial_{2,\delta}}$.

Bilan : ce complexe est exact.  
De plus, il résout $\bfB'_\delta$, car 
$\Im \widetilde {\partial_{1,\delta}} = \uPdelta + \bfA \nabla$ 
(la régularité de $\uP$ n'intervient pas ici).
\end{proof}

\medskip

Pour terminer, signalons que nous définirons (cf.~\ref{sous-sectionIdeauxFitting}) 
la notion  de rang \textit{d'un module de présentation finie} 
par l'intermédiaire de ses idéaux
de Fitting, ceux-ci étant les idéaux déterminantiels
(convenablement indexés) d'une présentation du module.  
Ainsi, lorsque
la suite $\uP$ est régulière, le rang du $\bfA$-module~$\bfB_d$ associé à
$\uP$ est l'entier $\chi_d$ caractérisé par:
$$
\calF_{\chi_d +1}(\bfB_d) = 0, \qquad
\calF_{\chi_d}(\bfB_d) \text{ est fidèle.}
$$

\cleardoublepage

\section{Deux cas d'école : l'idéal d'élimination est explicitement monogène}
\label{DeuxCasEcole}

Dans ce chapitre, d'une part pour le format $D = (1, \dots, 1, e)$ et d'autre part dans le cas
$n=2$, nous montrons, pour un système \emph {générique} $\uP$, la monogénéité de
l'idéal d'élimination:
$$
\ElimIdeal = \langle\calR\rangle
$$
Le scalaire $\calR$ exhibé est un générateur privilégié, rendu unique
par une propriété de normalisation: c'est le résultant $\calR = \Res(\uP)$
du système $\uP$.

\index{résultant}%

Pour ces cas particuliers, nous n'utilisons pas les théorèmes de
structure des résolutions libres finies; nous aurons
essentiellement besoin du théorème de Wiebe et de la notion de
profondeur~$\geqslant 2$.

\subsubsection*{Le cas de $n$ formes linéaires : leur résultant est leur déterminant}

Bien que le cas de $n$ formes linéaires $L_i$ soit pris en charge par le cas
d'école $D = (1, \dots, 1, e)$ avec $e=1$, nous souhaitons illustrer
la puissance du résultat de Wiebe (cf.~\ref{MiniWiebe}) sur cet
exemple.  Contrairement à ce que l'on pourrait croire, une preuve
directe de la monogénéité de l'idéal d'élimination dans ce cas n'est
pas si aisée. 
Dans cet exemple, le degré critique $\delta$ vaut $0$, 
la matrice bezoutienne du système $(L_1, \dots, L_n)$ est la matrice $\dsL \in \bbM_n(\bfA)$
dont la $j$\up{ème} colonne est la suite des $n$ coefficients de~$L_j$ ; 
et son déterminant $\nabla = \det \dsL$ est un élément de $\bfA$
(en accord avec le fait que $\nabla$ soit un polynôme homogène de degré $\delta$).

\begin{prop}[idéal d'élimination pour $n$ formes linéaires] \label{FormesLineairesFromWiebe}
Soit $\uL = (L_1, \ldots, L_n)$ une suite {\rm régulière} constituée de $n$ formes linéaires
de $\bfA[\uX]_1$ dont nous notons $\dsL \in \bbM_n(\bfA)$ la matrice.  

On a l'égalité d'idéaux de $\bfA[\uX]$ :
$$
\langle \uL \rangle^\sat \ = \ 
\langle \uL,\, \det \dsL \rangle
$$
En particulier, l'idéal d'élimination est engendré par $\det\dsL$:
$$
\ElimIdealL \ =\  \langle \det \dsL \rangle 
$$
\end{prop}

\index{idéal!d'élimination}

\begin{proof}
Nous allons appliquer~\ref{MiniWiebe} à la suite $\uL$ et 
au déterminant bezoutien $\det \dsL$.

\medskip

Pour montrer l'égalité d'idéaux de $\bfA[\uX]$, 
il suffit de montrer l'égalité des composantes homogènes 
en degré $d$ et ceci \textit{pour tout $d \in \bbN$}.
Or, d'après~\ref{MiniWiebe}, on a exactement cette égalité de composantes 
homogènes mais \textit{seulement pour $d \geqslant \delta$}.
Mais ici, $\delta = 0$, donc c'est gagné.

\medskip
Pour prouver le \og En particulier \fg{}, prenons la composante homogène de degré $0$ de 
$\langle \uL \rangle^\sat = \langle \uL,\, \det \dsL \rangle$.
Comme $\dsL$ est une suite de polynômes homogènes de degré $1$, on en déduit que 
$\langle \uL \rangle_0^\sat$ est le $\bfA$-module engendré par $\det \dsL$.
\end{proof}

\bigskip

Anticipons un peu et donnons une nouvelle preuve de l'égalité 
$\ElimIdealL = \langle \det \dsL \rangle$.
La forme linéaire identité $\mu = \id_\bfA$ sur $\bfA[\uX]_{\delta} = \bfA$ 
vérifie toutes les hypothèses du résultat à venir~\ref{MonogeneiteElimIdeal}.
Quant au déterminant bezoutien de $\uL$, c'est $\nabla = \det \dsL$;
on en déduit que l'idéal d'élimination est monogène engendré par $\mu(\nabla)$, 
c'est-à-dire par $\det \dsL$.

\subsection{Forme linéaire sur $\bfA[\protect\uX]_\delta$ induisant la monogénéité de l'idéal d'élimination} 

La proposition suivante, sous le couvert d'un système $\uP$ de déterminant bezoutien $\nabla$,
prend en prémisse une forme linéaire
$\mu : \bfA[\uX]_\delta \to \bfA$ ayant des propriétés mirifiques
et fournit comme conclusion la monogénéité explicite de l'idéal d'élimination.
Expliquons-nous sur le statut d'une telle forme linéaire.
Afin de pouvoir traiter le plus rapidement possible 
quelques exemples, nous avons choisi d'axiomatiser ou plutôt de dégager
les propriétés primordiales de la forme linéaire, qui prendra plus tard le nom de 
$\omegares : \bfA[\uX]_\delta \to \bfA$.
Pour un système $\uP = (P_1, \dots, P_n)$, cette forme linéaire sera
définie (cf.~\ref{PoidsNormalisationMacRae}) comme étant le pgcd fort
d'une famille (finie) de formes linéaires sur $\bfA[\uX]_\delta$ dites
déterminantales. Une telle forme déterminantale s'obtient 
à partir de la matrice de Sylvester $\Syl_\delta(\uP)$
en y sélectionnant $s_\delta = \dim \bfA[\uX]_\delta - 1$ colonnes $C_1, \dots, C_{s_\delta}$.
Chaque colonne s'identifie à un polynôme de $\bfA[\uX]_\delta$, ce qui permet de définir
la forme déterminantale en question par l'intermédiaire d'un déterminant d'ordre $\dim\bfA[\uX]_\delta$:
$$
\det(C_1, \dots, C_{s_\delta}, \sbullet) : \bfA[\uX]_\delta \to \bfA
$$
Parmi ces formes déterminantales, l'une d'entre elles est la forme
linéaire notée $\omega$ définie en~\ref{sectionFormeLineaireOmega}.

\medskip

Dans les deux cas d'école que nous allons traiter, la stratégie est
légèrement différente de la stratégie générale : nous proposons une
(la) forme linéaire bien particulière ayant les propriétés
mystérieuses de la proposition qui suit.  
Dans le cas du format $D = (1,\dots, 1,e)$ \idest{} d'un système $(L_1, \dots, L_{n-1}, P_n)$ où les $L_i$ sont linéaires, 
cette forme linéaire est l'évaluation en un
point $\xi := (\xi_1, \dots, \xi_n)$ déterminé à partir des $L_i$:
$$
\xi_i = \det(L_1, \dots, L_{n-1}, X_i)
$$
et porte le nom de $\evalxi$ 
(il n'y a donc pas de pgcd ici). 
Pour $n=2$, la forme linéaire en question est déterminantale, 
nous l'avons notée~$\omega$.
Dans ce cas, on a l'égalité remarquable $\omegares = \omega$.

\label{NOTA03-omegares}%
%
%

\begin{prop}[Forme linéaire mirifique] 
\label{MonogeneiteElimIdeal}
Soit $\mu \in \bfA[\uX]_\delta^\star$ une forme linéaire  ayant la propriété
$$
\forall\, F \in \bfA[\uX]_\delta, \qquad 
\mu(\nabla)F - \mu(F) \nabla \in \uPdelta 
$$
que nous abrégerons en disant \og $\mu$ est de Cramer en $\nabla$\fg{} ou bien 
\og $\mu$ a la propriété de Cramer en $\nabla$\fg.

\begin{enumerate}[\rm i)]
\item Le scalaire $\calR = \mu(\nabla)$ appartient à $\ElimIdeal$.

\item Si $\uP$ est régulière et $\Gr(\mu) \geqslant 1$, alors $\calR$ est régulier.

\item Si $\uP$ est régulière et $\Gr(\mu) \geqslant 2$, alors
$\calR$ est un générateur de l'idéal d'élimination:  
$\ElimIdeal = \langle \calR \rangle$.
\end{enumerate}
\end{prop}

\index{forme linéaire!mirifique (abstraction de $\omegares$)}%
\index{idéal!d'élimination}%
\index{propriété Cramer!1@en $\nabla$ d'une forme linéaire}%

\begin{proof}
i) La propriété de Cramer en $\nabla$ fournit en particulier :
$$
\forall\, F \in \bfA[\uX]_\delta, \qquad 
\calR\, F \in \langle \uP,\, \nabla \rangle_\delta 
$$
Or $X_i \nabla \in \langle \uP\rangle_{\delta+1}$ pour tout $i$.
Par conséquent, $X^\alpha \,\calR \in \langle \uP \rangle_{\delta+1}$ pour tout $|\alpha | = \delta+1$.

\smallskip

\noindent
ii) Soit $a \in \bfA$ tel que $a \mu(\nabla) = 0$.
Pour tout $F \in \bfA[\uX]_\delta$, on a 
$$
a \, \mu(\nabla) F \ - \ a \, \mu(F) \nabla \ \in \ \uPdelta
$$
d'où $a \mu(F) \nabla \in \uPdelta$.  La suite $\uP$ étant régulière,
on peut appliquer le point i) du théorème de Wiebe~\ref{MiniWiebe}, à
savoir $\big(\uP : \nabla\big)_0 = 0$.  On obtient $a \mu(F) = 0$
dans~$\bfA$, et ceci pour tout $F$ ; d'où $a \mu = 0$.  Comme
$\Gr(\mu) \geqslant 1$, on en déduit $a = 0$.

\smallskip

\noindent
iii) Pour alléger les notations, on raisonne dans le quotient $\bfB = \bfA[\uX]/\langle\uP\rangle$
en notant $\overline F$ la classe de~$F$.
Puisque $\uP$ est régulière, on a d'après~\ref{WiebeCech}-ii) :
$$
\uPsat_\delta / \langle\uP\rangle_\delta \ = \  \bfA \,\overline\nabla
$$
Pour $G \in \uPsat_\delta$, désignons par $[G]\strut_\nabla$ la coordonnée de $\overline G$ sur 
la base $\overline \nabla$.

\smallskip

Soit $a \in \ElimIdeal$.
Pour $F \in \bfA[\uX]_\delta$, on a dans~$\bfB_\delta$ (à gauche par hypothèse, à droite par définition) 
$$
\calR \, \overline F \ =\ \mu(F) \overline \nabla 
\qquad \text{et} \qquad 
a\, \overline F \ =\  [aF]\strut_\nabla \overline \nabla
$$
En multipliant la première égalité par $a$ et la deuxième par $\calR$, on obtient en identifiant sur 
la $\bfA$-base~$\overline \nabla$, l'égalité dans $\bfA$ :
$$
a \mu(F) \ =\  \calR [a F]\strut_\nabla
\qquad 
\text{a fortiori} 
\qquad 
\calR \mid a\mu(F)
$$
On a donc $\calR \mid a \mu(X^\alpha)$ pour tout $|\alpha | = \delta$.
Or $\Gr(\mu) \geqslant 2$ et $\calR$ est régulier d'après le point ii). 
D'après~\ref{NeutralisationGr2}, 
on en déduit que $\calR \mid a$, ce qui termine la preuve de $a \in \langle \calR \rangle$.
\end{proof}

\label{NOTA03-mu}
%
%

Lorsque la suite $\uP$ est régulière, une forme linéaire $\mu$ comme
dans l'énoncé précédent possède un caractère exceptionnel comme en
témoignent les propriétés suivantes ; de plus, sous le couvert d'une
hypothèse adéquate, elle est essentiellement unique.

\begin{prop}
Supposons $\uP$ régulière et $\mu : \bfA[\uX]_\delta \rightarrow \bfA$ 
vérifiant $\mu(\nabla) F - \mu(F) \nabla \in \uPdelta$.

\begin{enumerate}[\rm i)]
\item 
La forme linéaire $\mu$ est nulle sur $\uPdelta$, ce qui s'écrit encore 
$\mu \in \Ker \transpose{\Syl_\delta}$.

En particulier, pour deux déterminants bezoutiens $\nabla$ et $\nabla'$ de $\uP$, on a 
$\mu(\nabla) =\mu(\nabla')$.

\item 
Si $\mu$ et $\mu'$ sont deux telles formes linéaires
et $\nabla$ un déterminant bezoutien de $\uP$, on a l'implication:
$$
\mu(\nabla) = \mu'(\nabla) 
\quad \Rightarrow \quad
\mu = \mu'
$$

\item Si $\mu$ et $\mu'$ vérifient $\Gr(\mu) \geqslant 2$ et $\Gr(\mu') \geqslant 2$, 
alors $\mu' = \varepsilon \mu$ pour un certain inversible $\varepsilon \in \bfA$.
\end{enumerate}
\end{prop}

\begin{proof}
  
i) Soit $F \in \uPdelta$. Une conséquence de l'hypothèse est $\mu(F) \nabla \in \uPdelta$.
Puisque $\uP$ est régulière, on peut appliquer le théorème de Wiebe~\ref{MiniWiebe} dont
le point i) affirme que $(\uP : \nabla)_0 = 0$ donc $\mu(F) = 0$.
Le \og En particulier \fg{} résulte de~\ref{NablaDansLeSature} 
et du fait que $\mu$ s'annule sur $\uPdelta$.

\medskip
\noindent
ii) On raisonne modulo $\uPdelta$. Pour $F \in \bfA[\uX]_\delta$, 
on a $\mu(\nabla) F \equiv \mu(F) \nabla$, idem pour $\mu'$.
L'hypothèse $\mu(\nabla) = \mu'(\nabla)$ fournit 
$\mu(F)\nabla \equiv \mu'(F) \nabla$. 
On utilise de nouveau \ref{MiniWiebe}-i) \idest{} 
$(\uP : \nabla)_0 = 0$, donc $\mu(F) = \mu'(F)$.

\medskip
\noindent
iii) 
Comme $\Gr(\mu) \geqslant 2$ et $\Gr(\mu') \geqslant 2$, 
les scalaires $\mu(\nabla)$ et $\mu'(\nabla)$ sont deux générateurs 
réguliers de l'idéal monogène $\ElimIdeal$ d'après~\ref{MonogeneiteElimIdeal}.
Il existe donc un inversible $\varepsilon$ tel que $\mu'(\nabla) = \varepsilon  \mu(\nabla)$.
Ainsi, $\mu'$ et $\varepsilon \mu$ sont deux formes linéaires 
qui coïncident en un déterminant bezoutien, donc elles sont égales d'après~{ii).}
\end{proof}

Juste pour le plaisir, donnons une preuve très légèrement différente du point iii) 
de~\ref{MonogeneiteElimIdeal}, en utilisant le fait qu'une telle forme linéaire $\mu$
est nulle sur $\uPdelta$.
Prenons $a \in \ElimIdeal$. Alors pour tout $X^\alpha \in \bfA[\uX]_{\delta}$, 
on a $aX^\alpha \in \uPsat_\delta$.
De plus, $\uPsat_\delta = \uPdelta + \bfA \nabla$ (d'après~\ref{MiniWiebe}-ii), 
donc on a 
$$
\forall\, |\alpha| = \delta,\qquad 
aX^\alpha \ \in\ \uPdelta + \bfA \nabla
$$
En évaluant en $\mu$, qui est nulle sur $\uPdelta$,
on obtient $a\mu(X^\alpha) \in \bfA \mu(\nabla)$ \idest{} $\calR \mid a \mu(X^\alpha)$.
On termine de la même manière qu'en~\ref{MonogeneiteElimIdeal} pour obtenir $\calR \mid a$.

\subsection{Premier cas d'école : le cas $D=(1,\dots, 1,e)$}
\label{soussectionPremierCasEcole}

Ici $\uP = (L_1, \dots, L_{n-1},  G)$ avec $L_1, \dots, L_{n-1}$ 
des formes linéaires et $G$ homogène de degré $e \geqslant 1$.
On~a $\delta = e-1$; le \MoutonNoir{} est $X_n^{\delta}$ mais il n'interviendra
pratiquement pas.
Ici, puisque $\delta = \deg G -1$, on a:
$$
\langle\uP\rangle_\delta \ =\  \langle L_1, \ldots, L_{n-1}\rangle_\delta
$$ 
Soit $L \in \bbM_{n, n-1}(\bfA)$ la matrice des coefficients des $L_i$, c'est-à-dire 
la matrice (elle est unique) vérifiant :
$$
\begin{bmatrix}
L_1 &
\cdots &
L_{n-1} 
\end{bmatrix}
\ = \ 
\begin{bmatrix}
X_1 &
\cdots &
X_n
\end{bmatrix}
L
$$
On définit les scalaires $\xi_i \in \bfA$ comme étant les mineurs maximaux signés 
de $L$ par la formule suivante où les~$T_i$ sont des indéterminées :
$$
\det 
\begin{bmatrix}
 &  &  & T_1 \\
 & L &  & \vdots \\ 
 &  &  & T_n \\
\end{bmatrix}
\ = \ 
\xi_1 T_1 + \cdots + \xi_n T_n
$$
On a $L_j(\uxi) = 0$ comme on le voit en remplaçant dans le déterminant la
dernière colonne par la colonne~$j$.

\noindent
En écrivant $G(\uX) = \displaystyle \sum_{i=1}^n X_i G_i(\uX)$
avec chaque $G_i$ homogène de degré $\delta-1$, 
on fait apparaître une matrice bezoutienne $\dsV$  ci-dessous :
$$
\begin{bmatrix}
L_1 &
\cdots &
L_{n-1} &
G
\end{bmatrix}
\ = \ 
\begin{bmatrix}
X_1 &
\cdots &
X_n
\end{bmatrix}
\dsV
\qquad \text{où} \qquad
\dsV = \begin{bmatrix}
 &  &  & G_1 \\
 & L &  & \vdots \\ 
 &  &  & G_n \\
\end{bmatrix}
\label{EgaliteTrefle}
$$
Le déterminant $\nabla(\uX)$ de cette matrice carrée $\dsV$ est 
le polynôme $\sum_{i=1}^n \xi_i G_i(\uX)$. 
\label{NablaD11eDef}

Introduisons la forme linéaire 
$\evalxi : \bfA[\uX]_\delta \to \bfA$, définie par $F \mapsto F(\uxi)$. 
On dispose alors de l'égalité $\evalxi(\nabla) = \sum_{i=1}^n \xi_i G_i(\uxi) = G(\uxi)$. 

\label{NOTA03-uxi}%
\label{NOTA03-evalxi}%

\medskip
Pour chaque forme linéaire $L_j$, nous allons distinguer deux de ses coefficients.
Les voici pour $n=4$ (les autres sont masqués par {\footnotesize $*$} ) :
$$
L \ = \
\begin{bmatrix}
p_1 & {}_{*} & {}_{*}\\
q_1 & p_2 & {}_{*} \\
{}_{*} & q_2 & p_3 \\
{}_{*}  & {}_{*} & q_3 \\
\end{bmatrix}
$$

\begin{prop}[Propriétés fondamentales de la forme linéaire $\evalxi$]  
\label{evalxiProperties} 
\leavevmode
\begin{enumerate}[\rm i)]
\item
La forme linéaire $\evalxi$ est nulle sur $\langle\uP\rangle_\delta = \langle L_1,\dots,L_{n-1}\rangle_\delta$.
  
\item 
Pour tous $F, H \in \bfA[\uX]_\delta$, on a :
$$
\evalxi(H)F \, -\,  \evalxi(F)H \ \in \ \langle\uP\rangle_\delta
$$
En particulier, pour tout $F\in \bfA[\uX]_\delta$,
$$
\evalxi(\nabla)F \,-\, \evalxi(F)\nabla \ \in\  \langle\uP\rangle_\delta
$$

\item
Soit $\uP$ suffisamment générique au sens suivant : les $2(n-1)$ coefficients des deux \og grandes
diagonales\fg{} de la matrice $L$ sont des indéterminées sur un anneau de base $\bfR$. Il s'agit de
$$
p_i := \coeff_{X_i}(L_i) \quad \text{et} \quad
q_i := \coeff_{X_{i+1}}(L_i)  
\qquad \text{pour \ $1\leqslant i \leqslant n-1$}
$$
Alors la suite $(\xi_1^\delta, \xi_n^\delta)$ est régulière.
En particulier, $\Gr(\evalxi) \geqslant 2$.

\item 
Pour le jeu étalon, $\evalxi$ est la forme linéaire $(X_n^\delta)^\star$,
coordonnée sur le {\rm mouton-noir}.

\item 
Les coefficients de $\evalxi$ sont homogènes de poids $\delta = e-1$ en $L_i$ 
et de poids $0$ en $G$.
\end{enumerate}
\end{prop}

\index{poids (en $P_i$)}%
%
%

\begin{proof} \leavevmode

\noindent
i) Découle du fait que $L_j(\uxi) = 0$.

\noindent
ii) Appliquons le lemme de Cramer (cf.~\ref{CramerSymetrie}) au $\bfA$-module libre
$\bfA X_1 \oplus \cdots \oplus \bfA X_n$ de base $(X_1, \ldots, X_n)$
en posant $\Delta_{\uL}(F) = \det(L_1, \dots, L_{n-1}, F)$ pour $F$ linéaire.
On a $\Delta_{\uL}(X_i) X_j - \Delta_{\uL}(X_j) X_i \in \langle L_1, \ldots, L_{n-1}\rangle$ 
\idest{} 
$$
\xi_iX_j \equiv \xi_jX_i \bmod \langle L_1,\ldots, L_{n-1}\rangle
$$
On en déduit aussitôt:
$$
\xi_{i_1}\cdots \xi_{i_m} X_{j_1}\cdots X_{j_m} \equiv
\xi_{j_1}\cdots \xi_{j_m} X_{i_1}\cdots X_{i_m} \bmod \langle L_1,\ldots, L_{n-1}\rangle
$$
Par combinaison linéaire, pour tous $F,H \in \bfA[\uX]_\delta$:
$$
F(\uxi)\,H(\uX) - H(\uxi)\,F(\uX) \in \langle L_1, \ldots, L_{n-1}\rangle_\delta
$$

\noindent
iii) 
On va montrer que la suite
$\big(\evalxi(X_1^\delta),\ \evalxi(X_n^\delta)\big) = (\xi_1^\delta,\, \xi_n^\delta)$ est régulière ; 
il s'en suivra que $\Gr(\evalxi) \geqslant 2$. 
On va examiner les composantes homogènes dominantes de $\xi_1^\delta$ 
et~$\xi_n^\delta$ en tant qu'éléments de $\bfR[p_1, \dots, p_{n-1}, q_1, \dots, q_{n-1}]$.
On en fait un peu plus en étudiant $\xi_i^\delta$ pour tout $i$.

\medskip

On commence par déterminer $\xi_i = \evalxi(X_i^\delta)$ 
qui est le déterminant de la matrice dont les $n-1$ premières colonnes 
sont occupées par les coefficients de $L_1, \dots, L_{n-1}$ et la dernière 
colonne contient un~$1$ en $i$\up{ème} ligne et des $0$ ailleurs. 
Autrement dit :
$$
\xi_i \ = \ 
\det 
\begin{bmatrix}
&  &  &  & & 0 \\
&  &  &  & & \vdots \\
&  & L &  & & 1 \\ 
&  &  &  & & \vdots \\
&  &  &  & & 0 \\
\end{bmatrix}
\ = \ 
(-1)^{i+n} \det A_i
$$
où $A_i$ est la matrice carrée de taille $n-1$ ci-dessous
(la diagonale du bloc nord-ouest est constituée de $p_1, \dots, p_{i-1}$ 
et la sous-diagonale de $q_1, \dots, q_{i-2}$ ;
la diagonale du bloc sud-est est constituée de $q_i, \dots, q_{n-1}$ 
et la sur-diagonale de $p_{i+1}, \dots, p_{n-1}$ ;
tous les autres coefficients sont dans~$\bfR$) :
$$
A_i \ = \ 
\EastBordermatrix{
p_1 &            &            &          & \VR \\ 
q_1 & \ddots &            &          & \VR  \\ 
       & \ddots & \ddots  &          & \VR  \\ 
       &            & q_{i-2} & p_{i-1} & \VR  \\ 
 \HR{8}
& & & & \VR q_i & p_{i+1} & & & \\ 
 & & & & \VR & \ddots & \ddots & & \\ 
 & & & & \VR & & \ddots & p_{n-1} & \\ 
 & & & & \VR &  & & q_{n-1} & \\ 
}
$$
La composante homogène de degré $n-1$ de $\det A_i$ 
est réduite au terme $p_1 \cdots p_{i-1} q_i \cdots q_{n-1}$.
Les composantes homogènes de degré supérieur sont nulles.
Ainsi la composante homogène dominante de $\xi_i^\delta$ 
est le produit $(-1)^{(i+n)\delta} (p_1 \cdots p_{i-1} q_i \cdots q_{n-1})^\delta$.

En particulier, la composante homogène dominante de $\xi_1^\delta$ est 
$(-1)^{(1+n)\delta} (q_1 \cdots q_{n-1})^\delta$ 
et celle de $\xi_n^\delta$ est $(p_1 \cdots p_{n-1})^\delta$.
Ce sont deux polynômes qui forment une suite régulière.
D'après le contrôle de la profondeur par les composantes homogènes dominantes
(cf.~\ref{ControleProfCompHmgDom-nequal2}),  
on en déduit que la suite $(\xi_1^\delta, \xi_n^\delta)$ est régulière.

\smallskip

iv) 
Pour le jeu étalon, on a $\xi_n = 1$ et $\xi_i = 0$ pour $i < n$.
Ainsi $\evalxi$ est exactement la forme linéaire coordonnée $(X_n^\delta)^\star$.

\smallskip

v)
On voit aisément que $\xi_j$ est de poids $1$ en $L_i$ et $0$ en $G$.
Ainsi pour $|\alpha|=\delta$, l'élément $\evalxi(X^\alpha) = \xi^\alpha$ est de poids $\delta = e-1$ 
en les $L_i$ et de poids $0$ en $G$.
\end{proof}

\subsubsection{Monogénéité de l'idéal d'élimination pour $D = (1, \dots, 1,e)$}

\begin{theo}\label{Resultant11d}
Soit $\uP = (L_1, \dots, L_{n-1}, G)$ la suite générique de format 
$D = (1, \dots, 1,e)$ et $\nabla$ un déterminant bezoutien de $\uP$.

L'idéal d'élimination est monogène
$$
\ElimIdeal \ = \ \langle \calR \rangle
\qquad 
\text{avec\ $\calR = \evalxi(\nabla) = G(\uxi)$}
$$
Ce générateur $\calR$ est normalisé, dans le sens où il vaut $1$ pour le jeu étalon 
$(X_1, \dots, X_{n-1}, X_n^e)$.
De plus, il est homogène en $P_i$ : 
son poids en $P_n = G$ est $1$ et son poids en $L_i$ vaut $e$.
De manière uniforme, en notant $(d_1, \dots, d_n) = (1, \dots, 1, e)$,
le générateur $\calR$ est homogène en $P_i$ de poids $\widehat {d_i} := d_1 \cdots d_n/d_i$.
\end{theo}

\index{idéal!d'élimination}%
\index{poids (en $P_i$)}%
%
%

\begin{proof}
D'après~\ref{evalxiProperties}-ii, la forme linéaire $\mu = \evalxi$
possède en particulier la propriété \og de Cramer en~$\nabla$ \fg 
(lire la remarque~\ref{CramerTotalRmq} en ce qui concerne \og en particulier\fg).
Comme la suite $\uP$ est générique, elle est régulière et $\Gr(\mu) \geqslant 2$ 
en vertu de~\ref{evalxiProperties}-iii.
Ainsi, en appliquant~\ref{MonogeneiteElimIdeal}, on obtient que l'idéal 
d'élimination est monogène engendré par $\calR = \evalxi(\nabla)$.
Comme $\nabla(\uX) = \sum \xi_i G_i(\uX)$
(cf. la définition de $\nabla$ page~\pageref{NablaD11eDef}), 
on obtient $\evalxi(\nabla) = G(\uxi)$.

Ce générateur est normalisé. 
En effet, pour le jeu étalon, on a $\xi_n = 1$ et $G = X_n^e$, d'où 
$\calR = G(\uxi) = 1$.

\noindent
Concernant le poids, 
chaque $\xi_j$ est homogène de poids $1$ en les coefficients des $L_i$ et de poids $0$ en $G$.
Donc $\calR = G(\uxi)$ est de poids annoncé.

Pour la normalisation et le poids, 
on peut également s'appuyer sur $\calR = \evalxi(\nabla)$.
Pour la normalisation, on peut invoquer~\ref{NablaEtalon} 
accompagné de~\ref{evalxiProperties}-iv).
Quant au poids, on utilise le fait que le poids de $\nabla$ en $P_i$ est $1$
(c'est vrai pour tout déterminant bezoutien $\nabla$)
et on invoque~\ref{evalxiProperties}-v) pour la contribution de $\evalxi$.
\end{proof}

\begin{rmq}
Pour $D = (1,\dots, 1,e)$, 
la monogénéité de l'idéal d'élimination a lieu dès lors que le coefficient en~$X_i^{d_i}$ de~$P_i$
est une indéterminée pour tout~$i$ (ce qui assure que la suite~$\uP$ est régulière
d'après~\ref{JeuCouvrantEtalonRegularite}),
et que le coefficient en~$X_i$ et en~$X_{i+1}$ de~$L_i$ est une indéterminée 
pour tout ${i \leqslant n-1}$ 
(ce qui permet d'obtenir la profondeur~${\geqslant 2}$ d'après~\ref{evalxiProperties}).
\end{rmq}

\begin{rmq}[Le cas particulier de $n$ formes linéaires]
Dans le cas où $e =1$, c'est-à-dire $D = (1,\dots, 1)$ et $\uP = (L_1, \dots, L_{n-1}, L_n)$,
on~a $\delta = 0$ et la matrice carrée bezoutienne $\dsV$ qui intervient 
page~\pageref{EgaliteTrefle} est la matrice qui donne les coefficients 
des formes linéaires $L_i$.
Le déterminant~$\nabla$ de cette matrice est un élément de $\bfA$, donc évaluer 
ce polynôme en $\uxi$ n'est pas difficile (il n'y a rien besoin de faire, puisque 
le polynôme est un élément de $\bfA$). Ainsi $\evalxi(\nabla) = \nabla$. 
Le résultant est donc $\nabla$ et on retrouve~\ref{FormesLineairesFromWiebe}.
\end{rmq}

\begin{rmq}
Pour ce cas d'école $(L_1, \dots, L_{n-1}, P_n)$ où les $L_i$ sont linéaires,
voici une manière d'expliquer \textit{a posteriori} l'égalité
$$
\Res(L_1, \dots, L_{n-1}, P_n) 
\ = \ 
P_n(\uxi)
$$
Cela suppose évidemment que le résultant soit en place et que l'on en
connaisse quelques propriétés.  Parmi ces propriétés, il y a le fait
que $\Res(P_1, \dots, P_n)$ est homogène en $P_i$, de poids $\widehat
{d_i} = \prod_{j \ne i} d_j$. Ici, on a $\widehat d_n = 1$ de sorte
que $\Res(L_1, \dots, L_{n-1}, P_n)$ est \textit{linéaire} en $P_n$.
On a également besoin de la propriété de multiplicativité du résultant
(cf.~\ref{MultiplicativiteResultant}),
que l'on utilise sous la forme suivante où $X^\alpha = X_1^{\alpha_1} \cdots X_n^{\alpha_n}$:
$$
\Res(L_1, \dots, L_{n-1}, X^\alpha) 
\ = \ 
\prod_{i=1}^n
\Res(L_1, \dots, L_{n-1}, X_i)^{\alpha_i} 
$$
Comme les $n$ polynômes $L_1, \dots, L_{n-1}, X_i$  sont linéaires,
on a (cf.~\ref{FormesLineairesFromWiebe}) :
$$
\Res(L_1, \dots, L_{n-1}, X_i) 
\ = \ 
\det(L_1, \dots, L_{n-1}, X_i)
$$
Mais le déterminant à droite de cette égalité, c'est notre définition de $\xi_i$. Il vient alors en combinant les deux dernières égalités :
$$
\Res(L_1, \dots, L_{n-1}, X^\alpha)
\ = \ 
\prod_{i=1}^n \xi_i^{\alpha_i}
\ = \ 
\uxi^\alpha
$$
Et enfin, en utilisant qu'ici, dans ce cas d'école, le résultant est linéaire 
relativement au dernier polynôme, on obtient :
$$
\Res(L_1, \dots, L_{n-1}, P_n) 
\ = \ 
P_n(\uxi)
$$
et on retrouve bien la formule annoncée en~\ref{Resultant11d}.
\end{rmq}

\subsection{Second cas d'école : le cas $n=2$}

On note ici $X$ et $Y$ les deux indéterminées ainsi que $P$ et $Q$ les deux polynômes homogènes 
respectivement de degré $p$ et~$q$.
On a donc $\delta = (p-1) + (q-1)$.
Nous allons expliciter une forme linéaire $\omega : \bfA[X,Y]_\delta \rightarrow \bfA$ 
ayant les propriétés mirifiques de la proposition~\ref{MonogeneiteElimIdeal}.
Si on relit le commentaire venant avant cette proposition, 
cette forme linéaire devrait être notée $\omegares$, mais 
quand tout sera en place, on comprendra que dans le cas $n=2$, on a $\omegares = \omega$.
De manière approximative, sans se soucier de signe, $\omega$ est la forme linéaire constituée des 
mineurs signés de la matrice de Sylvester en degré $\delta$ 
de format $(\delta +1) \times \delta$. 
Cependant notre tâche, ici et plus loin, doit être accomplie sans trop abuser 
des matrices, afin de bien contrôler la normalisation en le jeu étalon.

\subsubsection{Le complexe de Koszul et ses composantes homogènes}

\index{complexe composante homogène $\rmK_{\sbullet,d}(\uP)$}%
\index{matrice!de Sylvester}%

Voici un petit complexe (c'est le complexe de Koszul de la suite $(P,Q)$ de l'anneau $\bfA[X,Y]$) :
$$
\xymatrix @M=0.4pc @C=4pc{
\bfA[X,Y] \ar[r]^-{ 
\left [
\begin{smallmatrix}
-Q \\ P
\end{smallmatrix}
\right]
}
& \bfA[X,Y] \times \bfA[X,Y] \ar[r]^-{%
\left [
\begin{smallmatrix}
P & Q
\end{smallmatrix}
\right]
}
& \bfA[X,Y] 
}
$$
La première différentielle est l'application de Sylvester 
$$
\Syl : \bfA[X,Y] \times \bfA[X,Y] \longrightarrow \bfA[X,Y], 
\qquad (U,V) \longmapsto UP + VQ
$$
Prenons la composante homogène de degré $d$ du complexe de Koszul, qui est un 
complexe de $\bfA$-modules
$$
\xymatrix @M=0.4pc @C=2pc{
\bfA[X,Y]_{d-(p+q)} \ar[r]
& \bfA[X,Y]_{d-p} \times \bfA[X,Y]_{d-q} \ar[r]^-{\Syl_d}
& \bfA[X,Y]_d
}
$$
On constate que pour $d < \delta + 2$, 
\idest{} pour $d \leqslant \delta+1$, 
ce complexe ne comporte qu'une seule différentielle, à savoir $\Syl_d$. 
Deux valeurs de $d$ nous intéressent plus particulièrement, il s'agit de $d =\delta$ et $d=\delta+1$~:
$$
\Syl_\delta : 
\bfA[X,Y]_{q-2} \times \bfA[X,Y]_{p-2} \rightarrow \bfA[X,Y]_{p+q-2}
\quad \text{et} \quad
\Syl_{\delta+1} : 
\bfA[X,Y]_{q-1} \times \bfA[X,Y]_{p-1} \rightarrow \bfA[X,Y]_{p+q-1}
$$
de sorte que $\Syl_\delta$ est représentée par une matrice de type Hilbert-Burch 
(une colonne de moins que le nombre de lignes, qui vaut d'ailleurs $p+q-1 = \delta+1$, 
confer~\ref{HilbertBurch}) 
et $\Syl_{\delta+1}$ est représentée par une matrice \textit{carrée} de taille $p+q$.

\`A partir de $\Syl_\delta$, on va construire une forme linéaire $\omega : \bfA[X,Y]_\delta \rightarrow \bfA$
et à partir de $\Syl_{\delta+1}$, on va définir un endomorphisme 
$W_{1, \delta+1} : \bfA[X,Y]_{\delta+1} \rightarrow \bfA[X,Y]_{\delta+1}$. 

\smallskip
Par endroits, nous aurons besoin de nommer la base canonique de 
$\bfA[X,Y]^2$ : 
nous avons choisi $(e_P, e_Q)$.

\subsubsection{Le degré $\delta$ pour $n=2$}

Pour $i+j=\delta = p+q-2$, 
on a $i \geqslant p$ ou $j \geqslant q$ sauf dans un cas, le cas où $i=p-1$ et $j=q-1$.
Cette remarque a valu au monôme $X^{p-1}Y^{q-1}$ l'appellation de \MoutonNoir.
De plus, le \og ou \fg{} est exclusif, car l'inégalité $i \geqslant p$ implique 
l'inégalité $j \leqslant q-1$.
En utilisant la notion d'idéal excédentaire de
la section~\ref{DefIdealModuleExcedentaire}, on résume cela par
les égalités ci-dessous qui ne font intervenir que le format de
degrés~$D=(p,q)$
$$
\bfA[X,Y]_\delta \ = \ 
\Jex_{1,\delta} \,\oplus\, \bfA . X^{p-1}Y^{q-1} 
\ = \ 
\langle X^p,\, Y^q \rangle_\delta \,\oplus\, \bfA . \mouton{}
\qquad \text{et} \qquad
\Jex_{2,\delta} = 0
$$

\begin{defn}[Le constructeur $\Omega(\ \cdot \ )$ et la forme linéaire $\omega$] \label{Defomegan=2}
\leavevmode

Pour $F \in \bfA[X,Y]_\delta$, on définit l'endomorphisme 
$\Omega(F) : \bfA[X,Y]_\delta \rightarrow \bfA[X,Y]_\delta$ par :
$$
\begin{array}{c}
\\
 \Omega(F) :\  X^iY^j \longmapsto \\
\\
\end {array}
\left\{
\begin{array}{ll}
X^{i-p}Y^j \, P & \text{si $i \geqslant p$} \\ [-0.3em]
& \\ 
X^{i}Y^{j-q} \, Q & \text{si $j \geqslant q$} \\ [-0.3em]
& \\ 
F & \text{si $i=p-1$ et $j=q-1$}
\end{array}
\right.
$$
\`A l'application $\Omega : \bfA[X,Y]_\delta \rightarrow \mathrm{End}(\bfA[X,Y]_\delta)$ 
est associée la forme linéaire \og déterminant \fg{} suivante :
$$
\omega : 
\begin{array}[t]{rcl}
\bfA[X,Y]_\delta & \longrightarrow & \bfA \\ 
F & \longmapsto & \det \big(\Omega(F) \big)
\end{array}
$$
\end{defn}

\label{NOTA03-Omega}%
\label{NOTA03-omega}%

Par exemple, prenons $p = 3$ et $q = 5$.
Notons 
$$
P  \,= \, a_{0}X^{3} + a_{1}X^{2}Y + a_{2}XY^{2} + a_{3}Y^{3}
\qquad \quad
Q \, =\,  b_{0}X^{5} + b_{1}X^{4}Y + b_{2}X^{3}Y^{2} + b_{3}X^{2}Y^{3} + b_{4}XY^{4} + b_{5}Y^{5}
$$
Pour $F = \displaystyle \sum_{i = 0}^6 f_i X^i Y^{6-i}$ de degré $\delta = 6$,
l'endomorphisme $\Omega(F)$ est représenté par la matrice :
$$
\Omega(F) \ = \ 
\EastBordermatrix{
a_{0} & . & . & . & f_0 & b_{0} & . & \Heti{X^{6}} \\ 
a_{1} & a_{0} & . & . & f_1 & b_{1} & b_{0} & \Heti{X^{5}Y} \\ 
a_{2} & a_{1} & a_{0} & . & f_2 & b_{2} & b_{1} & \Heti{X^{4}Y^{2}} \\ 
a_{3} & a_{2} & a_{1} & a_{0} & f_3 & b_{3} & b_{2} & \Heti{X^{3}Y^{3}} \\ 
. & a_{3} & a_{2} & a_{1} & f_4 & b_{4} & b_{3} & \Heti{X^{2}Y^{4}}  \mouton \\ 
. & . & a_{3} & a_{2} & f_5 & b_{5} & b_{4} & \Heti{XY^{5}} \\ 
. & . & . & a_{3} & f_6 & . & b_{5} & \Heti{Y^{6}} \\ 
}
$$
Autrement dit, $\omega$ est la \og forme linéaire des mineurs signés\fg{} de la matrice 
de $\Syl_\delta$ (cette dernière matrice n'a de sens que si l'on a précisé 
les bases au départ et à l'arrivée, car il s'agit d'une application linéaire, et non d'un endomorphisme) :
$$
\Syl_{\delta} \ = \
\NorthEastBordermatrix{
\Veti{X^{3}\, e_P} & \Veti{X^{2}Y\, e_P} & \Veti{XY^{2}\, e_P} & \Veti{Y^{3}\, e_P} 
& \Veti{X\, e_Q} & \Veti{Y\,e_Q} & \\
a_{0} & . & . & . & b_{0} & . & \Heti{X^{6}} \\
a_{1} & a_{0} & . & . & b_{1} & b_{0} & \Heti{X^{5}Y} \\
a_{2} & a_{1} & a_{0} & . & b_{2} & b_{1} & \Heti{X^{4}Y^{2}} \\
a_{3} & a_{2} & a_{1} & a_{0} & b_{3} & b_{2} & \Heti{X^{3}Y^{3}} \\
. & a_{3} & a_{2} & a_{1} & b_{4} & b_{3} & \Heti{X^{2}Y^{4}} \\
. & . & a_{3} & a_{2} & b_{5} & b_{4} & \Heti{XY^{5}} \\
. & . & . & a_{3} & . & b_{5} & \Heti{Y^{6}} \\
}
$$
Le cas $n=2$ est spécial, dans le sens où, dans $\Syl_\delta$, 
il y a une colonne de moins que de lignes. 
C'est ce phénomène qui donne naturellement naissance, au signe près, à une forme linéaire
dite des mineurs signés, dont les coefficients sont des mineurs pleins de $\Syl_\delta$.
Reste le problème du signe.
Ici, dans le cadre du résultant, on peut rigidifier cette forme linéaire en décidant de la normaliser 
par le jeu étalon $(X^p, Y^q)$ : cette forme linéaire doit prendre la valeur $1$ en~$X^{p-1}Y^{q-1}$.
Et il n'y a qu'une seule façon de réaliser cela : mettre la colonne-argument 
à la même position que le \MoutonNoir{} $X^{p-1}Y^{q-1}$ au sein des monômes de 
degré $\delta$. C'est exactement ce qui a été imposé dans la définition 
de $\Omega(F)$.

\index{forme linéaire!des mineurs signés}
%
%

\begin{prop}[Propriétés fondamentales de la forme linéaire $\omega$ pour $n=2$] 
\label{omegaProperties-n=2}
\leavevmode

\noindent
Soit $\nabla$ un déterminant bezoutien de $(P,Q)$, cf.~\ref{DefNabla}.

\begin{enumerate}[\rm i)]
\item
La forme $\omega : \bfA[X,Y]_\delta \to \bfA$ est nulle sur $\langle P,Q\rangle_\delta$.  

\item 
Pour tous $F, G \in \bfA[X,Y]_\delta$, on a :
$$
\omega(G)F \, -\,  \omega(F)G \ \in \ \langle P, Q \rangle_\delta
$$
En particulier, pour tout $F\in \bfA[X,Y]_\delta$,
$$
\omega(\nabla)F \,-\, \omega(F)\nabla \ \in\  \langle P, Q \rangle_\delta
$$

\item
On suppose que $P$ et $Q$ s'écrivent
$$
P \ = \ u_0 X^p + u_1 X^{p-1}Y + \cdots
\qquad \text{et\ } \qquad
Q \ = \ v_0 Y^q + v_1 Y^{q-1}X + \cdots
$$
où les quatre coefficients $u_0, u_1, v_0, v_1$ sont des indéterminées 
sur un anneau $\bfR$, les autres coefficients étant dans $\bfR$, 
de sorte que $\bfA = \bfR[u_0, u_1, v_0, v_1]$.

La suite $\big(\omega(X^\delta),\ \omega(Y^\delta)\big)$ est régulière ; 
en particulier, $\Gr(\omega) \geqslant 2$.

\item 
Pour le jeu étalon, $\omega$ est la forme linéaire coordonnée 
sur le {\rm mouton-noir}, $(X^{p-1}Y^{q-1})^\star$.

\item 
Les coefficients de $\omega$ sont homogènes de poids $q-1$ en les coefficients de $P$
et de poids $p-1$ en les coefficients de $Q$.
\end{enumerate}
\end{prop}

\index{poids (en $P_i$)}%
%
%

\begin{proof}  \leavevmode

\noindent
i) Un système de générateurs de $\langle P,Q\rangle_\delta$ est
constitué des $q-1$ polynômes $(X^{q-2-j}Y^jP)_{0\le j\le q-2}$ et des
$p-1$ polynômes $(X^iY^{p-2-i}Q)_{0\le i\le p-2}$.  La matrice
$\Omega(X^{q-2-j}Y^jP)$ possède 2 colonnes identiques: sa colonne
mouton-noir $X^{p-1}Y^{q-1}$ et sa colonne $X^{\delta-j}Y^j$, donc son
déterminant est nul.  Idem pour la matrice $\Omega(X^iY^{p-2-i}Q)$
dont la colonne mouton-noir est égale à la colonne $X^iY^{\delta-i}$.

\smallskip

\noindent
ii) Ci-après $i$ et $j$ sont contraints par l'égalité $i+j = \delta$.
Equipons le $\bfA$-module libre $\bfA[X,Y]_\delta$ de la base monomiale 
$(X^iY^j)_{i \geqslant p} \vee (X^iY^j)_{j \geqslant q} \vee X^{p-1}Y^{q-1}$.
La forme linéaire $\omega$ s'écrit alors $\Delta_{\uv}$ 
où $\uv$ est la suite 
$\big( (X^{i-p}Y^j P)_{i \geqslant p} \vee (X^iY^{j-q}Q)_{j \geqslant q}\big)$.
On applique Cramer à $\Delta_{\uv}$, cf.~\ref{CramerSymetrie}.

\smallskip

\noindent
iii) On va montrer que la suite $\big(\omega(X^\delta),\ \omega(Y^\delta)\big)$ 
est régulière ; 
on aura alors ${\Gr\big(\omega(X^\delta),\ \omega(Y^\delta)\big) \geqslant 2}$, 
a fortiori $\Gr(\omega) \geqslant 2$.
Pour cela, on va raisonner sur les composantes homogènes 
dominantes (qui se trouvent être de degré $\delta$) de $\omega(X^\delta)$  et $\omega(Y^\delta)$
en tant que polynômes en $u_0, u_1, v_0, v_1$.

Par définition, $\omega(X^\delta)$ est le déterminant de l'endomorphisme $\Omega(X^\delta)$.
Voici sa matrice dans la base de $\bfA[X,Y]_\delta$ spécifiée ci-dessous. Cette matrice 
est de taille $\delta +1$ et son déterminant est égal à $(-1)^{q-1}$ fois le déterminant 
de la matrice obtenue en rayant la première ligne et la colonne \MoutonNoir{}. 
Cette dernière matrice, de taille $\delta$, 
a des coefficients de degré $\leqslant 1$ en les indéterminées. 
Les composantes homogènes de degré $> \delta$ de son déterminant sont nulles.
Pour en obtenir la composante homogène de degré $\delta$ en $u_0,u_1,v_0,v_1$, 
on considère, dans le développement du déterminant de taille $\delta$,
les produits contenant uniquement ces indéterminées.
Il n'y en a qu'un seul, c'est $u_1^{q-1} v_0^{p-1}$ (qui est de degré~$\delta$).
Bref, la composante homogène dominante de $\omega(X^\delta)$ est 
égale à $(-1)^{q-1} u_1^{q-1} v_0^{p-1}$.
$$
\Omega(X^\delta) \ = \ 
\EastBordermatrix{
u_0 & & & & & 1 &\VR & & & & &   \Heti{X^\delta} \\ 
u_1 & u_0 & & & & 0 &\VR & & & & & \Heti{X^{\delta-1} Y} \\
 & u_1 & \ddots & &  & \vdots & \VR & & & & & \omit \qquad \vdots\\
 & & \ddots & \ddots & & \vdots & \VR & & & & &\\
 & & & \ddots & u_0 & \vdots &\VR & & & & &\Heti{X^{p} Y^{q-2}} \\
 & & & & u_1 & 0 & \VR v_1 & & & & & \omit \quad \mouton \\
 \HR{11}
& & & & & 0 & \VR v_0 & v_1 & & & & \Heti{X^{p-2}Y^q} \\
& & & & &  &\VR & v_0 & \ddots & & & \omit \qquad \vdots \\
& & & & & \vdots &\VR &  & \ddots & \ddots & & \\ 
& & & & &  &\VR & & & \ddots & v_1 & \Heti{XY^{\delta-1}} \\
& & & & & 0 &\VR & & & & v_0 & \Heti{Y^{\delta}} \\
}
$$
De la même façon, on voit que la composante homogène dominante 
de $\omega(Y^\delta)$ vaut $(-1)^{p-1} v_1^{p-1} u_0^{q-1}$.
$$
\Omega(Y^\delta) \ = \ 
\EastBordermatrix{
u_0 & & & & & \VR 0& & & & & &   \Heti{X^\delta} \\ 
u_1 & u_0 & & & & \VR  & & & & & & \Heti{X^{\delta-1} Y} \\
 & \ddots & \ddots & &  &\VR \vdots & & & & & & \omit \qquad \vdots\\
 & & \ddots & \ddots & & \VR  & & & & & &\\
 & & & u_1 & u_0 & \VR 0 & & & & & &\Heti{X^{p} Y^{q-2}} \\
 \HR{11}
 & & & & u_1 & \VR  0 & v_1 & & & & & \omit \quad \mouton \\
& & & & &\VR  \vdots & v_0 & v_1 & & & & \Heti{X^{p-2}Y^q} \\
& & & & & \VR \vdots & & v_0 & \ddots & & & \omit \qquad \vdots \\
& & & & & \VR \vdots & &  & \ddots & \ddots & & \\ 
& & & & & \VR 0 & & & & \ddots & v_1 & \Heti{XY^{\delta-1}} \\
& & & & &\VR 1 & & & & & v_0 & \Heti{Y^{\delta}} \\
}
$$
La suite $\big( u_1^{q-1} v_0^{p-1}, \ v_1^{p-1} u_0^{q-1}\big)$ est
régulière.
D'après~\ref{ControleProfCompHmgDom-nequal2},  
on en déduit que la suite $\big( \omega(X^\delta), \omega(Y^\delta)\big)$ est régulière.

\smallskip

iv) Pour le jeu étalon, on a $\Omega(X^{p-1}Y^{q-1}) = \Id$ donc $\omega(X^{p-1}Y^{q-1}) = 1$.
Et on a $\omega(X^\alpha) = 0$ pour $X^\alpha \neq X^{p-1}Y^{q-1}$.

\smallskip

v) 
Cela résulte de la construction de $\Omega$.
Le nombre de couples $(i,j)$ tels que $i+j=\delta = p+q-2$ et $i \geqslant p$ 
est exactement égal à $q-1$. Il y a donc $q-1$ colonnes dont les 
coefficients sont soit des $0$ soit un coefficient de $P$.
Idem pour $Q$ en échangeant $p$ et $q$.
\end{proof}

\subsubsection{Monogénéité de l'idéal d'élimination pour $n=2$}

\begin{theo}\label{EliminationThFor2}
Soit $(P,Q)$ la suite générique de format $D = (p,q)$ de $\bfA[X,Y]$ 
et $\nabla$ un déterminant bezoutien de $(P,Q)$, cf.~\ref{DefNabla}.

L'idéal d'élimination est monogène
$$
\langle P,Q \rangle^\sat \cap \bfA 
\ = \ \langle \calR \rangle
\qquad 
\text{avec\ $\calR = \omega(\nabla)$}
$$
Ce générateur $\calR$ est normalisé dans le sens où il vaut $1$ pour le jeu étalon 
$(X^p, Y^q)$.
De plus, c'est un polynôme homogène de poids $q$ en les coefficients de $P$ 
et homogène de poids $p$ en les coefficients de $Q$.
\end{theo}

\index{idéal!d'élimination}%
%
%

\begin{proof}
D'après~\ref{omegaProperties-n=2}-ii, la forme linéaire $\omega$
possède la propriété \og de Cramer en $\nabla$ \fg{}, et même beaucoup
plus (cf. la remarque \ref{CramerTotalRmq}).
Comme la suite~$(P,Q)$ est générique, elle est régulière et $\Gr(\omega) \geqslant 2$ 
en vertu de~\ref{omegaProperties-n=2}-iii.
Ainsi, en appliquant~\ref{MonogeneiteElimIdeal}, on obtient que l'idéal 
d'élimination est monogène engendré par $\calR = \omega(\nabla)$.
Ce générateur est normalisé d'après~\ref{omegaProperties-n=2}-iv
et~\ref{NablaEtalon}.
L'élément $\calR$ est de poids attendu en~$P$ : 
en effet, le poids de $\nabla$ est $1$ et celui de $\omega$ est~$q-1$
d'après~\ref{omegaProperties-n=2}-v. 
Idem pour le poids en $Q$.
\end{proof}

\subsubsection{En degré $\delta+1$ pour $n=2$}

Pour $i+j=\delta+1 = p+q-1$, 
on a $i \geqslant p$ ou $j \geqslant q$.
De plus, le \og ou \fg{} est exclusif, car l'inégalité $i \geqslant p$ implique 
l'inégalité $j \leqslant q-1$.
En utilisant les notations de la section \ref{DefIdealModuleExcedentaire},
on résume cela par :
$$
\bfA[X,Y]_{\delta+1} \ = \ 
\Jex_{1,\, \delta+1} \ = \ 
\langle X^p,\, Y^q \rangle_{\delta+1}
\qquad \text{et} \qquad
\Jex_{2,\,\delta+1} = 0
$$

\begin{defn} \label{DefW1deltaplus1nequal2}
On définit l'endomorphisme 
$W_{1,\,\delta+1}: \bfA[X,Y]_{\delta+1} \rightarrow \bfA[X,Y]_{\delta+1}$ par :
$$
\begin{array}{c}
\\
W_{1,\,\delta+1} :\ X^iY^j \longmapsto \\
\\
\end {array}
\left\{
\begin{array}{ll}
X^{i-p}Y^j \, P  & \text{si $i \geqslant p$} \\ [-0.3em]
\text{\footnotesize \rm ou} & \\ 
X^{i}Y^{j-q} \, Q & \text{si $j \geqslant q$} 
\end{array}
\right.
$$
\end{defn}

\label{NOTA03-W1delta+1}%

Par exemple pour $p = 3$ et $q = 5$, on a $\delta+1 = 7$ et la matrice de $W_{1, \delta+1}$ est de 
taille $p+q = 8$.
Avec 
$$
P  \,= \, a_{0}X^{3} + a_{1}X^{2}Y + a_{2}XY^{2} + a_{3}Y^{3}
\qquad \quad
Q \, =\,  b_{0}X^{5} + b_{1}X^{4}Y + b_{2}X^{3}Y^{2} + b_{3}X^{2}Y^{3} + b_{4}XY^{4} + b_{5}Y^{5}
$$
on a :
$$
W_{1, 7} \ = \ 
\EastBordermatrix{
a_{0} & . & . & . & . & b_{0} & . & . & \Heti{X^{7}} \\ 
a_{1} & a_{0} & . & . & . & b_{1} & b_{0} & . & \Heti{X^{6}Y} \\ 
a_{2} & a_{1} & a_{0} & . & . & b_{2} & b_{1} & b_{0} & \Heti{X^{5}Y^{2}} \\ 
a_{3} & a_{2} & a_{1} & a_{0} & . & b_{3} & b_{2} & b_{1} & \Heti{X^{4}Y^{3}} \\ 
. & a_{3} & a_{2} & a_{1} & a_{0} & b_{4} & b_{3} & b_{2} & \Heti{X^{3}Y^{4}} \\ 
. & . & a_{3} & a_{2} & a_{1} & b_{5} & b_{4} & b_{3} & \Heti{X^{2}Y^{5}} \\ 
. & . & . & a_{3} & a_{2} & . & b_{5} & b_{4} & \Heti{XY^{6}} \\ 
. & . & . & . & a_{3} & . & . & b_{5} & \Heti{Y^{7}} \\ 
}
$$

\begin{prop}\label{detWdelta+1InElimIdeal}
Soit $(P,Q)$ la suite générique.
Le scalaire $\det W_{1,\, \delta+1}$ appartient à l'idéal d'élimination 
$\langle P,Q \rangle^\sat \cap \bfA$. 
Il est normalisé, dans le sens où pour le jeu étalon $(X^{p}, Y^{q})$, il vaut~$1$.
De plus, c'est un polynôme homogène de poids $q$ en les coefficients de $P$ 
et homogène de poids~$p$ en les coefficients de $Q$.
\end{prop}

\begin{proof}
On va raisonner matriciellement.
Ci-dessous, $i+j = \delta+1 = p+q-1$.
La réunion disjointe $(X^iY^j)_{i \geqslant p}  \vee (X^iY^j)_{j \geqslant q}$ est une partition de 
l'ensemble de tous les monômes de degré $\delta+1$.
On a l'égalité matricielle suivante 
où les matrices lignes sont à coefficients dans $\bfA[X,Y]$ et la matrice carrée $W_{1, \delta+1}$
à coefficients dans $\bfA$ :
$$
\Big[\ 
(X^iY^j)_{i \geqslant p} 
\ \mid \  
(X^iY^j)_{j \geqslant q}  
\ \Big]
\ W_{1, \delta+1}
\quad = \quad
\Big[\ 
\big(X^{i-p} Y^j \,P \big)_{i \geqslant p}
\ \mid \ 
\big(X^{i} Y^{j-q} \, Q \big)_{j \geqslant q}
\  \Big]
$$
Par conséquent, en multipliant à droite par la transposée de la comatrice 
et en notant $\calR' = \det W_{1,\delta+1}$, 
on obtient $X^i Y^j \calR' \in \langle P,Q\rangle$ pour tous $i+j = \delta+1$.
Cela signifie que $\calR'$ est dans l'idéal d'élimination.

Le fait que $\calR'$ soit normalisé provient du fait que l'endomorphisme $W_{1,\delta+1}$ 
est lui-même normalisé, \idest{} est égal à l'identité pour le jeu étalon (c'est évident d'après la définition).

Le fait que $\calR'$ soit homogène en les coefficients de $P$ 
est dû au fait que les coefficients d'une colonne de la matrice de $W_{1,\delta+1}$ 
sont ou bien tous homogènes de poids $1$ en les coefficients de $P$ 
ou bien tous de poids $1$ en les coefficients de $Q$.
Pour la précision sur le poids, il suffit de remarquer que 
les $(i,j)$ tels que $i +j = \delta+1$ et $i \geqslant p$ sont au nombre de $q$ 
et les $(i,j)$ tels que $i +j = \delta+1$ et $j \geqslant q$ sont au nombre de~$p$. 
\end{proof}

\subsubsection{Lien entre $W_{1,\delta+1}$ et \og la bonne vieille \fg{} 
matrice de Sylvester de 2 polynômes en une indéterminée}

Pour $n=2$ et $d=\delta+1 = p+q-1$, l'application de Sylvester est :
$$
\Syl_{\delta+1} : \
\bfA[X,Y]_{q-1} \, e_P \oplus \bfA[X,Y]_{p-1}\, e_Q 
\ \longrightarrow \ \bfA[X,Y]_{p+q-1} 
$$
Il est remarquable que les espaces de départ et d'arrivée aient même dimension
(à savoir $p+q$). Mieux, ils sont monomialement isomorphes via 
$$
\begin {array}{c}
 \\
 \varphi : \ X^i Y^j \rightarrow \\
 \\
\end {array}
\left\{
\begin{array}{ll}
X^{i-p}Y^j \, e_P  & \text{si $i \geqslant p$} \\ [-0.3em]
\text{\footnotesize ou} & \\ 
X^{i}Y^{j-q} \, e_Q & \text{si $j \geqslant q$} 
\end{array}
\right.
$$
On remarque que l'on a $W_{1,\delta+1} = \Syl_{\delta+1} \circ \varphi$.
D'où le résultat suivant pour le cas $n=2$.

\begin{prop}
Soit $\calB_{p+q-1}$ la base monomiale de $\bfA[X,Y]_{p+q-1}$.
La matrice de l'application linéaire $\Syl_{\delta+1}$ 
dans les bases $\varphi(\calB_{p+q-1})$ au départ et 
$\calB_{p+q-1}$ à l'arrivée est égale à celle de 
l'endomorphisme $W_{1,\delta+1}$ dans la base $\calB_{p+q-1}$.
\end{prop}

Par exemple, si l'on écrit 
$\calB_{p+q-1} = X^p\calB_{q-1} \vee Y^q\calB_{p-1}$ 
(on met d'abord les monômes de degré $p+q-1$ 
divisibles par $X^p$ puis ensuite ceux divisibles par $Y^q$), alors 
$\varphi(\calB_{p+q-1}) = \calB_{q-1} e_P \vee \calB_{p-1}e_Q$.
Si l'on prend maintenant pour $\calB_{q-1} = (X^{q-1}, X^{q-2}Y, \dots, Y^{q-1})$ (idem pour $\calB_{p-1}$), 
alors $\calB_{p+q-1} = (X^{p+q-1}, X^{p+q-2}Y, \dots, Y^{p+q-1})$. 
En faisant $X := T$ et $Y := 1$, on obtient les bases 
$$ 
\calB_{p+q-1} = (T^{p+q-1}, \dots, T, 1)
\qquad 
\text{et\ }
\qquad
\varphi(\calB_{p+q-1}) = (T^{q-1} e_P, \dots, T e_P, e_P) \vee 
(T^{p-1} e_Q, \dots, T e_Q, e_Q)
$$
La matrice de l'application
$$
\bfA[T]_{\leqslant q-1} \, e_P \oplus \bfA[T]_{\leqslant p-1} \, e_Q
\ \longrightarrow \  \bfA[T]_{\leqslant p+q-1}, 
\qquad \qquad 
Ue_P + V e_Q \ \longmapsto \ UP+VQ
$$
dans les bases précédentes est alors la matrice de Sylvester 
\og habituelle\fg{} pour deux polynômes en une seule indéterminée
$P = a_0 T^p + a_1 T^{p-1}+ \dots + a_p$ et $Q = b_0 T^q + b_1 T^{q-1} + \dots + b_q$ :
$$
\NorthEastBordermatrix{
\Veti{T^{q-1}P} & \Veti{T^{q-2}P} & \cdots & \cdots &\Veti{TP} & \Veti{P} & 
\Veti{T^{p-1}Q} & \cdots & \Veti{Q} &  \cr 
a_0 &  &  &  &  & & b_0 & & & \Heti{T^{p+q-1}} \cr
\vdots & a_0 & & & &  & \vdots & \ddots & & \Heti{\vdots} \cr
\vdots & \vdots & \ddots & & &  & &  & b_0 & \Heti{T^{q}} \cr
a_p & \vdots & & \ddots& &  & \vdots &  & \vdots & \Heti{T^{q-1}} \cr
 & a_p & & & &  & & & &  \Heti{\vdots} \cr
 & & \ddots & & & a_0 & \vdots & & \vdots & \Heti{T^{p}} \cr
 & & & \ddots & & \vdots & b_q & & &  \Heti{T^{p-1}} \cr
 & & & & & \vdots & & \ddots & \vdots &  \Heti{\vdots} \cr
  & & & & & a_p & &  & b_q & \Heti{1} \cr
}
$$
Cette matrice est celle de $W_{1,\delta+1}$ exprimée 
dans la base $\calB_{p+q-1}= (X^{p+q-1}, X^{p+q-2}Y, \dots, Y^{p+q-1})$.

\begin{theo}[Où l'on relie le degré $\delta$ au degré $\delta+1$ pour $n=2$] \label{Resultantn=2}

\leavevmode

On a l'égalité remarquable
$$
\omega(\nabla) \ = \ 
\det W_{1, \delta+1}
$$
Autrement dit, l'idéal d'élimination $\langle P,Q \rangle^\sat \cap \bfA$ 
est engendré par $\det W_{1, \delta+1}$.
\end{theo}

\begin{proof}
Il suffit de vérifier cette identité algébrique en terrain générique. On suppose donc que 
$P$ et~$Q$ sont génériques.
D'après~\ref{EliminationThFor2}, 
l'idéal d'élimination $\langle P,Q \rangle^\sat \cap \bfA$ est engendré par 
$\calR = \omega(\nabla)$.
Or $\calR'= \det W_{1, \delta+1}$ est dans l'idéal d'élimination 
(cf.~\ref{detWdelta+1InElimIdeal}).
Par conséquent, il existe $\lambda \in \bfA = \bfk[\text{indets pour $P$ et $Q$}]$ tel que 
$\calR' = \lambda \calR$.
Or $\calR$ et $\calR'$ sont des polynômes homogènes en les coefficients de $P$ 
et de~$Q$ de même poids. Ainsi $\lambda$ est une constante de $\bfk$. 
En spécialisant en le jeu étalon et en invoquant le fait que ces deux éléments sont normalisés, 
on obtient $\lambda = 1$, c'est-à-dire $\calR' = \calR$.
\end{proof}

Signalons une autre preuve moins structurelle de ce dernier résultat.
Cette preuve utilise le développement de Laplace. 
Intervient la notion de \emph {signe étendu} défini de la manière suivante pour $i,j \in \bbZ$
et $I,J$ deux parties finies de $\bbZ$ (ou plus généralement d'un ensemble
totalement ordonné quelconque):
$$
\varepsilon(i,j) = 
\begin {cases} 0 &\text{si } i=j\\ 1 &\text{si } i<j\\ -1 &\text{si } i>j\\ \end {cases}
\qquad\qquad
\varepsilon(I,J) = \prod_{\substack{i\in I\\ j \in J}} \varepsilon(i,j)
$$

\index{signe étendu}%
\label{NOTA03-epsilon}%
%
%

\begin{proof}[Autre preuve]
On vient de fournir une preuve structurelle de l'identité algébrique $\omega(\nabla) = 
\det W_{1,\delta+1}$. Ici, on la montre via un développement de Laplace, 
tout en l'illustrant sur l'exemple $p=3$, $q=5$ :
$$
P = a_{0}X^{3} + a_{1}X^{2}Y + a_{2}XY^{2} + a_{3}Y^{3}
\qquad
Q = b_{0}X^{5} + b_{1}X^{4}Y + b_{2}X^{3}Y^{2} + b_{3}X^{2}Y^{3} + b_{4}XY^{4} + b_{5}Y^{5}
$$
On {\it fixe} la matrice suivante $\dsV$ telle que $[P,Q] = [X,Y]\, \dsV$:
$$
\dsV =
\left[
\begin{array}{*{2}{c}}
a_{0}X^{2} + a_{1}XY + a_{2}Y^{2} \quad & \quad
b_{0}X^{4} + b_{1}X^{3}Y + b_{2}X^{2}Y^{2} + b_{3}XY^{3} + b_{4}Y^{4} \\ 
a_{3}Y^{2}& b_{5}Y^{4} \\ 
\end{array}
\right]
$$
et son déterminant $\nabla$ vaut :
$$
\nabla =
-a_{3}b_{0}X^{4}Y^{2} -a_{3}b_{1}X^{3}Y^{3} + (a_{0}b_{5} - a_3b_2)X^{2}Y^{4} +
(a_{1}b_{5} - a_3b_3)XY^{5} + (a_{2}b_{5} - a_3b_4)Y^{6}
$$
\`A gauche, la matrice $\Omega(\nabla)$ de taille $\delta +1 = p+q-1$ 
et à droite, la matrice $W_{1,\delta+1}$ de taille $\delta+2 = p+q$ :
$$
\EastBordermatrix{
a_{0}&  . &  . &  . &  . & b_{0}&  . & \Heti{X^6} \\ 
a_{1}& a_{0}&  . &  . &  . & b_{1}& b_{0} & \Heti{X^5Y}\\ 
a_{2}& a_{1}& a_{0}&  . & -a_{3}b_{0}& b_{2}& b_{1} & \Heti{X^4Y^2} \\ 
a_{3}& a_{2}& a_{1}& a_{0}& -a_{3}b_{1}& b_{3}& b_{2} & \Heti{X^3Y^3} \\ 
 . & a_{3}& a_{2}& a_{1}& a_{0}b_{5} -a_{3}b_{2}& b_{4}& b_{3} & \Heti{X^2Y^4} \\ 
 . &  . & a_{3}& a_{2}& a_{1}b_{5} -a_{3}b_{3}& b_{5}& b_{4} & \Heti{XY^5} \\ 
 . &  . &  . & a_{3}& a_{2}b_{5} -a_{3}b_{4}&  . & b_{5} & \Heti{Y^6} \\ 
}
\qquad \qquad 
\EastBordermatrix{
a_{0} & . & . & . & . & b_{0} & . & . & \Heti{X^{7}} \\ 
a_{1} & a_{0} & . & . & . & b_{1} & b_{0} & . & \Heti{X^{6}Y}\\
a_{2} & a_{1} & a_{0} & . & . & b_{2} & b_{1} & b_{0} & \Heti{X^{5}Y^{2}}\\
a_{3} & a_{2} & a_{1} & a_{0} & . & b_{3} & b_{2} & b_{1} & \Heti{X^{4}Y^{3}}\\
. & a_{3} & a_{2} & a_{1} & a_{0} & b_{4} & b_{3} & b_{2} & \Heti{X^{3}Y^{4}}\\
. & . & a_{3} & a_{2} & a_{1} & b_{5} & b_{4} & b_{3} & \Heti{X^{2}Y^{5}}\\ 
. & . & . & a_{3} & a_{2} & . & b_{5} & b_{4} & \Heti{XY^{6}} \\ 
. & . & . & . & a_{3} & . & . & b_{5} & \Heti{Y^{7}} \\ 
}
$$
Pour le calcul du déterminant de gauche, on développe le long de la colonne 
\MoutonNoir{} $X^{p-1}Y^{q-1}$ d'indice $q$ ; on a :
$$
\det \Omega(\nabla) \ = \ 
\sum_{i=1}^{\delta+1}
\ (-1)^{i+q} 
\ \nabla_{X^{\delta-(i-1)} Y^{i-1}}
\, \times \, \det \Omega(\nabla)_{\overline i \times \overline q}
\leqno (\star)
$$
où $\Omega(\nabla)_{\overline i \times \overline q}$ est la matrice 
issue de $\Omega(\nabla)$ obtenue en supprimant la ligne $i$ et la colonne $q$.

\smallskip
\noindent
Pour le déterminant de droite, on développe  avec Laplace 
le long des deux colonnes fixes $J_0 = (q,p+q)$ : 
la colonne~$q$
correspond à la dernière colonne de $P$ (les $a_i$) et la colonne~$p+q$ à la
dernière colonne de $Q$ (les $b_j$) :
$$
\det W_{1,\delta+1} 
\ = \ 
\varepsilon(J_0,\overline{J_0}) 
\sum_{\#I=2} \varepsilon(I,\overline I) 
\ \det W_{I \times J_0} 
\, \times \, 
\det W_{\overline I \times \overline {J_0}}
$$
Notons $I_i = (i, p+q)$ pour $1 \leqslant i \leqslant p+q-1$.
Tout d'abord, montrons que l'on peut sommer sur ces indices de lignes $I_i$ 
uniquement. Si $I \neq I_i$, c'est que $I$ ne contient pas $p+q$, donc ne 
s'appuie pas sur la dernière ligne.
Dans ce cas, soit la sous-matrice $W_{I \times J_0}$ a sa première ligne nulle, 
soit $W_{\overline I \times \overline {J_0}}$ a sa dernière ligne nulle.
On a donc
$$
\det W_{1,\delta+1} 
\ = \ 
\varepsilon(J_0,\overline{J_0}) 
\sum_{i=1}^{p+q-1} \varepsilon(I_i,\overline I_i) 
\ \det W_{I_i \times J_0} 
\, \times \, 
\det W_{\overline I_i \times \overline {J_0}}
\leqno (\star')
$$
Maintenant, nous sommes en mesure de comparer les deux sommes 
qui apparaissent en $(\star)$ et $(\star')$.
Pour tout $1\leqslant i\leqslant p+q-1$, 
on a l'égalité entre le coefficient $(i,q)$ de $\Omega(\nabla)$
et le mineur $I_i \times J_0$ de $W_{1,\delta+1}$, c'est-à-dire 
$\nabla_{X^{\delta-(i-1)} Y^{i-1}} = \det W_{I_i \times J_0}$ : 
$$
\text{la ligne $i$ de } 
\left[
\begin{array}{*{1}{c}}
 .  \\ 
 .  \\ 
-a_{3}b_{0} \\ 
-a_{3}b_{1} \\ 
a_{0}b_{5} -a_{3}b_{2} \\ 
a_{1}b_{5} -a_{3}b_{3} \\ 
a_{2}b_{5} -a_{3}b_{4} \\ 
\end{array}
\right]
\mbox{est égale au mineur $I_i$ de}
\left|
\begin{array}{*{2}{c}}
.& . \\ 
.& . \\ 
.& b_{0} \\ 
.& b_{1} \\ 
a_{0}& b_{2} \\ 
a_{1}& b_{3} \\ 
a_{2}& b_{4} \\ 
a_{3}& b_{5} \\ 
\end{array}
\right|
$$
Par exemple, pour $i = 4, 6,7$, on a :
$$
-a_3b_1 
\ \overset{4}{=}\ 
\begin{vmatrix} . & b_1\\ a_3 & b_5\\ \end{vmatrix},
\qquad \qquad 
a_1b_5-a_3b_3 
\ \overset{6}{=}\ 
\begin{vmatrix} a_1 & b_3\\ a_3 & b_5\\ \end{vmatrix},
\qquad \qquad 
a_2b_5-a_3b_4 
\ \overset{7}{=}\ 
\begin{vmatrix} a_2 & b_4\\ a_3 & b_5\\ \end{vmatrix}
$$
Quant aux mineurs complémentaires, 
on a l'égalité des matrices extraites (elles sont carrées de taille $\delta$):
$$
\Omega(\nabla)_{\overline i \times \overline q} 
\ = \ 
W_{\overline{I_i} \times \overline{J_0}} 
$$
Enfin, occupons-nous des signes. On a
$\varepsilon(J_0, \overline{J_0}) = (-1)^{p-1}$ et 
$\varepsilon(I_i, \overline{I_i}) = (-1)^{p+q-1+i}$ ;
le produit vaut~$(-1)^{i+q}$, signe qui figure dans le développement de $\det \Omega(\nabla)$.
\end{proof}

\begin {rmq} \label{CramerTotalRmq}
La lectrice attentive aura certainement remarqué que les deux formes linéaires $\mu : \bfA[\uX]_\delta \to \bfA$,
celle pour le format $D = (1,\dots, 1, e)$ de nom $\evalxi$, l'autre pour $n=2$ nommée $\omega$, possèdent
(cf.~\ref{evalxiProperties} et \ref{omegaProperties-n=2})
la propriété forte de Cramer suivante:
$$
\forall\ F,G \in \bfA[\uX]_\delta, \qquad 
\mu(G)F \, -\,  \mu(F)G \ \in \ \langle\uP \rangle_\delta 
\leqno (\star)
$$
Elle est bien plus forte que celle requise en $\nabla$ dans
\ref{MonogeneiteElimIdeal} pour que la forme linéaire soit qualifiée
de mirifique. L'explication est la suivante: ces deux formes linéaires
sont des cas particuliers de la forme linéaire $\omegares =
\omegaresP$ qui sera définie dans le chapitre \ref{ChapMacRaeForP}.  Celle-ci a le bon goût
de vérifier $(\star)$ quelque soit le jeu $\uP$.

\index{propriété Cramer!2@d'une forme linéaire}%

Nous aurions donc pu imposer $(\star)$ dans la proposition \ref{MonogeneiteElimIdeal} mais tel n'a pas
été notre choix qui s'est limité à la propriété Cramer en $\nabla$.

\medskip

Inutile de le cacher plus lontemps: pour n'importe quel système $\uP$,
dans la suite, nous allons faire la part belle aux formes linéaires
$\bfA[\uX]_\delta \to \bfA$ nulles sur $\langle\uP\rangle_\delta$
\idest{} aux formes linéaires sur $\bfB_\delta$ où, comme d'habitude,
$\bfB$ désigne le quotient de $\bfA[\uX]$ par $\langle\uP\rangle$.
Lorsque $\uP$ est régulière, le $\bfA$-module $\bfB_\delta$ est de
rang~$1$ (cf.~\ref{ResolutionQuotients}).
On verra que $\bfB_\delta$ se comporte \og localement \fg{}
comme un module \textit{libre} de rang 1, localement ayant ici le
sens \og après localisation en chaque élément d'une famille de 
profondeur $\geqslant 1$\fg{},
cf. la proposition \ref{FittingLocalisation}.  
On en déduit que deux formes linéaires $\alpha, \beta$ sur $\bfA[\uX]_\delta$, nulles sur
$\langle\uP\rangle_\delta$, sont proportionnelles au sens suivant
$$
\forall\ F,G \in \bfA[\uX]_\delta, \qquad
\alpha(F) \beta(G) = \alpha(G) \beta(F)
$$
En admettant cette égalité de proportionnalité, le lecteur peut, en
guise d'exercice, commencer à réfléchir aux compléments suivants, voire
les prouver puisque les ingrédients nécessaires sont présents dans ce
chapitre \og Deux cas d'école \fg{}.

Voici le contexte. Supposons disposer, comme en début de ce
chapitre, d'une forme linéaire $\mu : \bfA[\uX]_\delta \to \bfA$ ayant
la propriété de Cramer en $\nabla$ et vérifiant $\Gr(\mu) \geqslant 2$.
Alors la forme linéaire $\mu$ est une base de $\Ker \transpose
{\Syl_\delta}$.  De manière équivalente un peu plus abstraite, le
$\bfA$-module $\bfB_\delta^\star$ est libre de rang 1, de
base~$\overline\mu$.  De plus, l'évaluation en le bezoutien de $\uP$
est un isomorphisme:
$$
\begin{array}[t]{rcl}
\bfB_\delta^\star & \longrightarrow & \Ann(\bfB'_\delta) \\ [0.2cm]
\overline{\alpha} & \longmapsto & {\overline \alpha}(\overline {\nabla})
\end{array}
\qquad \text{ ou encore } \qquad 
\begin{array}[t]{rcl}
\Ker \transpose \Syl_\delta & \longrightarrow & \ElimIdeal \\ [0.2cm]
\alpha & \longmapsto & \alpha(\nabla)
\end{array}
$$
L'isomorphisme réciproque de celui de droite est 
$a \mapsto [a \times \sbullet]\strut_\nabla$
où pour $F \in \uPsat_\delta$, le symbole $[F]\strut_\nabla$ désigne la
coordonnée de $\overline F \in \bfB_\delta$ sur $\overline\nabla$, cf. la
preuve de la proposition~\ref{MonogeneiteElimIdeal}.
Nous y reviendrons dans un chapitre ultérieur (c'est l'objet
du théorème \ref{4pointsPsuperreguliere}, la forme $\mu$ ayant été dévoilée auparavant
sous son vrai nom $\omegares$).
\end {rmq}

\cleardoublepage

\section{Le jeu étalon généralisé $(p_1 X_1^{d_1}, \dots, p_n X_n^{d_n})$ et sa combinatoire}
\label{ChapJeuEtalonGeneralise}

Ici nous souhaitons traiter un cas d'école permettant, entre autres, d'illustrer la notion
d'invariant de MacRae.  Nous avons choisi l'exemple du \textit{jeu étalon généralisé} 
$\pXD = (p_1X_1^{d_1}, \dots, p_nX_n^{d_n})$ où $p_1,
\dots, p_n$ sont~$n$ indéterminées sur un petit anneau de base~$\bfk$.
L'anneau des coefficients ainsi obtenu est noté $\bfA = \bfk[p_1, \dots, p_n]$.

\medskip

Ce cas particulier paraît bien élémentaire, mais nous
en verrons plus tard l'impact sur un jeu quelconque $(P_1, \dots,
P_n)$ de format $D = (d_1, \dots, d_n)$. Il permet également
de dégager une certaine combinatoire, d'une part sur l'ensemble des monômes
(cf le rôle des idéaux monomiaux $\Jex_1,\Jex_2$ dans la suite),
d'autre part sur les mineurs (mécanisme de sélection, etc.),
combinatoire régie uniquement par le format de degrés (et non par le système $\uP$).

\index{jeu!e@étalon généralisé}%
\index{format de degrés}%
\index{invariant de MacRae}%

\medskip

Vouloir placer un tel exemple à cet endroit de notre étude nous
oblige à anticiper sur un certain nombre de notions qui seront
abordées en détails plus tard.  Nous avons fait ce choix en étant
conscients de provoquer certaines répétitions.  Le lecteur qui le
souhaite pourra se reporter aux définitions générales du
chapitre~\ref{ChapFittingVectoriel}.

\medskip

Fixons un format de degrés $D = (d_1, \dots, d_n)$ et rappelons (cf
chapitre \ref{ObjetsSuiteP}, en particulier
\ref{DefIdealModuleExcedentaire}) le rôle joué par le degré
critique $\delta = \sum_{i=1}^n (d_i-1)$ et l'idéal monomial $\Jex_1 =
\langle X_1^{d_1}, \dots, X_n^{d_n} \rangle$:
$$
\bfA[\uX]_d 
\ = \ 
\left\{
\begin{array}{ll}
\langle X_1^{d_1}, \dots, X_n^{d_n} \rangle_d & \text{pour $d \geqslant \delta+1$} \\
[0.2cm]
\langle X_1^{d_1}, \dots, X_n^{d_n} \rangle_{\delta} \oplus \bfA X^{\emouton}  
& \text{pour $d = \delta$} 
\end{array}
\right.
$$
La première égalité dit qu'un monôme $X^\alpha$ avec $|\alpha| \geqslant \delta+1$ 
est nécessairement divisible par (au moins) un $X_i^{d_i}$.
Et la seconde met en évidence le monôme $X^{\emouton} = X_1^{d_1 - 1} \cdots X_n^{d_n - 1}$
que nous avons baptisé \MoutonNoir{} \mouton:
c'est le seul monôme de degré $\delta$ non divisible par $X_i^{d_i}$ quelque soit~$i$.

\index{degré critique d'un système}%
\index{MoutonNoir@{\textit{mouton-noir}}}%

\medskip
En tout degré $d$, le $\bfA$-module 
$\bfA[\uX]_d$ possède une base privilégiée (non ordonnée): celle
constituée des monômes.  Ce qui permet de définir, pour n'importe quel
monôme $X^\alpha$ de degré $d$, la forme linéaire $(X^\alpha)^\star :
\bfA[\uX]_d \to \bfA$, coordonnée sur $X^\alpha$ relativement à cette
base monomiale.  En particulier, en degré~$\delta$, nous disposons de
la forme linéaire $(X^{\emouton})^\star$ coordonnée sur le \MoutonNoir.

\medskip
Traiter le cas particulier du jeu étalon généralisé n'a de sens que si
nous l'intégrons dans plan général de notre étude d'une suite
régulière $\uP = (P_1, \dots, P_n)$ de format $D$. Encore une fois,
nous allons insister sur la barrière $d \ge \delta+1$ et $d = \delta$.

\subsubsection*{Un aperçu de l'étude générale}

L'application de Sylvester est par définition une présentation du $\bfA[\uX]$-module $\bfB = \bfA[\uX]/\langle
\uP \rangle$: 
$$
\Syl = \Syl(\uP) : \, 
\xymatrix @C=1.5cm @M=0.4pc{
\bfA[\uX]^n \ar[r]^-{
\left [
\begin{smallmatrix}
P_1 & \cdots & P_n
\end{smallmatrix}
\right]
} & \bfA[\uX] 
}
$$
et induit pour chaque $d$ une présentation du $\bfA$-module $\bfB_d$, 
composante homogène de degré $d$ de $\bfB$:
$$
\Syl_d : \rmK_{1,d} = \bigoplus_{i=1}^n \bfA[\uX]_{d - d_i}\,e_i \longrightarrow \bfA[\uX]_d
$$
\index{application de Sylvester}%
Lorsque l'on prend le jeu étalon $\uX^D = (X_1^{d_1}, \dots, X_n^{d_n})$ comme suite $\uP$, 
le module $\bfB_d = \bfA[\uX]_d/\Jex_{1,d}$ est \textit{libre} de rang $\chi_d = \chi_d(D)$ où
$\chi_d = \dim \bfA[\uX]_d - s_d \buildrel {\rm def}\over = \dim \bfA[\uX]_d - \dim \Jex_{1,d}$.
Les idéaux déterminantiels de l'application de Sylvester du jeu étalon~$\uX^D$ vérifient
$$
\calD_{s_d+1}\big(\Syl_d(\uX^D)\big) = 0 \qquad \text{et}  \qquad
1 \in \calD_{s_d}\big(\Syl_d(\uX^D)\big)
$$
Lorsque la suite $\uP$ est régulière, l'application
de Sylvester $\Syl_d$ de $\uP$ possède comme rang l'entier $s_d$ au sens où:
$$
\calD_{s_d+1}(\Syl_d) = 0 
\qquad \text{tandis que } \qquad 
\calD_{s_d}(\Syl_d) \text{ est fidèle.}
$$
\index{rang!d'une application linéaire}%
La théorie des résolutions libres (finies) permettra de justifier
cette dernière affirmation; il ne s'agit donc pas d'un résultat que
l'on peut considérer comme élémentaire.

Ce dernier point s'énonce sous une forme équivalente
en termes du $\bfA$-module $\bfB_d$ associé à $\uP$: celui-ci
est de rang $\chi_d$, ce qui se mesure via ses idéaux de Fitting
introduits dans le chapitre ultérieur~\ref{ChapFittingVectoriel}
$$
\calF_{\chi_d +1}(\bfB_d) = 0 
\qquad \text{et } \qquad 
\calF_{\chi_d}(\bfB_d) \text{ est fidèle.}
$$
\index{rang!d'un module}%
Un certain nombre d'autres propriétés importantes vont être
conséquences de la théorie des résolutions libres.  En voici un
résumé.

\medskip
$\blacktriangleright$
En degré $d \geqslant \delta+1$, degré pour lequel $\Syl_d$ est de
rang maximum $s_d = \dim \bfA[\uX]_d$, l'idéal déterminantiel
$\calD_{s_d}(\Syl_d)$ possède un pgcd fort: c'est un scalaire régulier
de $\bfA$ n'appartenant pas en général à l'idéal. Dans la
suite, au lieu d'attribuer cette qualité uniquement à l'application
linéaire $\Syl_d$, nous l'attacherons également au $\bfA$-module
$\bfB_d$, conoyau de $\Syl_d$ : l'idéal déterminantiel précité prendra
l'appellation d'idéal de Fitting $\calF_0(\bfB_d)$.  Le $\bfA$-module
$\bfB_d$ sera alors déclaré de MacRae de rang $0$ et nous parlerons de
son invariant de MacRae noté $\MacRae(\bfB_d)$: c'est l'idéal engendré
par ce pgcd fort. On doit comprendre que le $0$ qui intervient ici est
la valeur de $\chi_d$.

\index{module!de MacRae}%
\index{invariant de MacRae}%

\medskip

Au fur et à mesure de notre étude, nous montrerons que pour $d
\geqslant \delta + 1$, chaque idéal monogène $\MacRae(\bfB_d)$
possède un générateur privilégié~$\calR_d$, \emph {indépendant} de $d$.
Ce générateur privilégié, \fbox {c'est le résultant de $\uP$}, disons
que c'est une définition possible.

\smallskip
Cette étude en degré $\geqslant \delta + 1$, attachée au rang $0$,
s'accompagne de celle qui suit en degré $d = \delta$ relevant du rang~1.

\medskip
$\blacktriangleright$
Pour $d = \delta$, on a $\chi_\delta = 1$ \idest{} $s_\delta =
\dim\bfA[\uX]_\delta -1$.  Ce ne sont pas les mineurs de $\Syl_\delta$
d'ordre maximum $\dim\bfA[\uX]_\delta$ que nous allons considérer car
ils sont nuls, mais une certaine structuration des mineurs
d'ordre~$s_\delta$.  Nous verrons en \ref{soussection-d-egal-delta}
que tout choix de $s_\delta$ colonnes de $\Syl_\delta$ permet de
construire une forme linéaire sur $\bfA[\uX]_\delta$ dont les
coefficients sont des mineurs d'ordre $s_\delta$.  Ces formes
déterminantales pures engendrent un sous-module
$\DVect_{s_\delta}(\Syl_\delta)$ de $\bfA[\uX]_\delta^\star$,
sous-module qui possède également la propriété d'avoir un pgcd
fort. Cette fois, ce pgcd fort n'est plus un scalaire mais une forme
linéaire notée $\omegares : \bfA[\uX]_\delta \to \bfA$, qui n'a aucune
raison d'appartenir à $\DVect_{s_\delta}(\Syl_\delta)$.  Cette forme
linéaire passe au quotient par $\bfA[\uX]_\delta \twoheadrightarrow
\bfB_\delta$ et définit ainsi une forme linéaire sur $\bfB_\delta$,
notée $\uomegares$ pour signaler le passage au quotient.

\medskip
A partir de ces divers objets définis à l'aide de $\Syl_\delta$, il
est préférable, pour des raisons structurelles, de former des
invariants associés au $\bfA$-module $\bfB_\delta$, parmi lesquels le
sous-module $\FittVect_1(\bfB_\delta) \subset \bfB_\delta^\star$
correspondant à $\DVect_{s_\delta}(\Syl_\delta)$ et ayant $\uomegares$ comme
pgcd fort. On décrétera le $\bfA$-module $\bfB_\delta$ de MacRae de
rang~1, d'invariant de MacRae $\MacRaeVect(\bfB_\delta)$, sous-module
engendré par la forme linéaire $\uomegares$.

\index{invariant de MacRae}%

\medskip

$\blacktriangleright$
Ces deux approches en degré $d \ge \delta+1$ et en degré $d = \delta$
sont reliées par l'égalité $\calR_d = \omegares(\nabla)$ où~$\nabla$
est un déterminant bezoutien de $\uP$. C'est le $\bfA$-module
$\bfB'_\delta = \bfB_\delta/\langle\nabla\rangle$ de rang~$0$ qui
permet d'assurer ce rapprochement.

\medskip
$\blacktriangleright$
Pour tout $d$, on dispose en fait d'un traitement uniforme.  Au
$\bfA$-module $\bfB_d$, on peut attacher son sous-module de Fitting
$\FittVect_{\chi_d}(\bfB_d) \subset \bigwedge^{\chi_d}(\bfB_d)^\star$,
qui possède un pgcd fort que nous
avons appelé l'invariant de MacRae \textit{vectoriel} de $\bfB_d$
et qui est une forme $\chi_d$-linéaire alternée sur $\bfB_d$.

\index{invariant de MacRae}%

\medskip
Ceci termine l'évocation d'un certain nombre de résultats,
qui, répétons-le, ne sont absolument pas évidents: ils nécessitent en
particulier un théorème dit de structure multiplicative des
résolutions libres.

\medskip
Nous pouvons cependant illustrer ces notions en prouvant les résultats
de manière directe dans le cadre du jeu étalon généralisé 
en degrés $d \geqslant \delta+1$ et $d = \delta$.

\smallskip
Voilà maintenant le théorème visé. 
Nous utilisons le vocabulaire qui sera mis en place ultérieurement 
en~\ref{DefModuleMacRae}
(et~\ref{DefMacRae0} pour le cas particulier du rang $0$), 
à savoir \textit{module de MacRae}.
Chaque point de ce théorème sera repris dans les propositions suivantes 
et prouvé à cette occasion.


\begin{theo}
Soit $\pXD = (p_1 X_1^{d_1},\dots, p_n X_n^{d_n})$ le jeu étalon généralisé 
et $\bfB = \bfA[\uX]/\langle \pXD\rangle$.
Posons $\widehat d_i = \prod_{j\neq i} d_j$.

\medskip

Pour $d\geqslant \delta+1$, le $\bfA$-module $\bfB_d$ est de MacRae de rang $0$ et 
son invariant de MacRae $\MacRae(\bfB_d)$ est le sous-module de $\bfA$ engendré par le scalaire
$\calR_d := p_1^{\widehat d_1}\cdots p_n^{\widehat d_n}$.

\smallskip
Pour $d = \delta$, le $\bfA$-module $\bfB_\delta$ est de MacRae de rang $1$ 
et son invariant de MacRae $\MacRaeVect(\bfB_\delta)$ est le sous-module de $\bfB_\delta^\star$ 
engendré par la forme linéaire 
$\omegares := p_1^{\widehat d_1 -1}\cdots p_n^{\widehat d_n -1} (X^\emouton)^\star$
passée au quotient.

\smallskip
On a $\omegares(\nabla) = \calR_d$ pour n'importe quel déterminant  bezoutien $\nabla$ de $\pXD$, par
exemple celui donné par $\nabla = p_1 \cdots p_n X^\emouton$.
\end{theo}

\label{NOTA04-ResJeuEtalonGen}
\label{NOTA04-omegaresJeuEtalonGen}
\label{NOTA04-MacRaeBd}
\label{NOTA04-MacRaeBdelta}

Dans ce chapitre, contrairement au chapitre précédent, l'idéal
d'élimination n'interviendra pas sauf dans la remarque importante qui
suit. Importante car elle annonce le statut de la définition à venir
du résultant comme un invariant de MacRae. Et non comme générateur de
l'idéal d'élimination, cette vision étant valide essentiellement dans
le cas générique.

\begin {rmq} [MacRae versus idéal d'élimination]\label{rmqMacRaeVsElimIdeal}
  
Le scalaire commun $\calR_d$ déterminé dans le théorème précédent,
invariant de MacRae de $\pXD = (p_1 X_1^{d_1},\dots,p_n X_n^{d_n})$,
c'est le résultant de $\pXD$ tel qu'il sera défini dans le chapitre
\ref{ChapMacRaeDefResultant}:
$$
\Res(\pXD) 
\ = \ 
p_1^{\widehat d_1}\cdots p_n^{\widehat d_n}
$$
Attention au fait qu'ici, ce n'est pas un générateur de l'idéal d'élimination. Ce dernier
idéal est bien monogène mais pas engendré par $\Res(\pXD)$:
$$
\langle \pXD \rangle^{\sat} \cap \bfA \ = \ 
\langle p_1 \cdots p_n \rangle
$$
Justifions cette égalité. Soit $a \in \langle \pXD \rangle^{\sat} \cap \bfA$;
pour chaque $i$, il y a un $e$ tel que 
$aX_i^e \in \langle p_1 X_1^{d_1},\dots,p_n X_n^{d_n}\rangle$,
donc $aX_i^e$ est multiple de $p_iX_i^{d_i}$. Ainsi $a$ est multiple de $p_i$,
et ceci pour tout $i$, d'où $a \in \langle p_1 \cdots p_n \rangle$.
L'autre inclusion est évidente.
\end {rmq}

\subsection{Le cas $d \geqslant \delta+1$}
\label{soussectionJeuEtalonGenDeltaPlus1}

Il va être question ici et ailleurs de mineurs d'une matrice.
Attention à cette terminologie \og mineur \fg{} qui peut être prise
dans plusieurs sens. En principe, elle désigne la suite des positions
des lignes/colonnes dans la matrice et conduit à une valeur bien
précise du déterminant (et pas au signe près); en toute rigueur, le
scalaire obtenu devrait être nommé \textit{valeur du mineur}.  Il nous
arrivera cependant, quand cela ne prête pas à confusion, de nous
écarter de ce principe et de désigner par mineur un déterminant défini
au signe près.  A contrario, aux endroits qui nous semblent délicats,
nous nous efforcerons d'apporter les précisions qui s'imposent.

\index{mineur}%

\medskip

L'expression \textit{mineur plein} d'une matrice désignera un mineur d'ordre égal au nombre de lignes.

\index{mineur!plein}%

\medskip
Pour le jeu étalon généralisé, en notant $\widehat d_i =
\prod_{j\ne i} d_j$, nous allons montrer que l'idéal des mineurs
pleins de~$\Syl_d$ possède $\calR := p_1^{\widehat d_1} \cdots
p_n^{\widehat d_n}$ comme pgcd fort. Comme annoncé dans
l'introduction, on dira que~$\bfB_d$ est un module de MacRae de rang
$0$ et que $\calR$, pgcd fort de l'idéal de Fitting $\calF_0(\bfB_d)$,
est son invariant de MacRae.

\medskip
Avant de traiter la technique, commençons par visualiser un exemple \og numérique \fg{}.

\subsubsection*{Exemple pour $n=3$, avec le format $D = (2,1,3)$ et le degré $d = \delta+1 = 4$}

\label{Exemple213-d4}

Voici les dimensions de l'espace de départ et d'arrivée de $\Syl_d : \rmK_{1,d} \to \bfA[\uX]_d$ :
$$
\dim \rmK_{1,d} = \sum_{i=1}^n \binom{d-d_i+n-1}{n-1} =
\binom {4}{2} + \binom {5}{2} + \binom {3}{2} = 19
\qquad
\dim \bfA[\uX]_d = \binom{d+n-1}{n-1} = \binom {6}{2} = 15
$$
Le nombre de mineurs pleins de $\Syl_d$ est $ \binom{19}{15} =
\frac {19 \times 18 \times 17 \times 16}{4!} = 3876 $.  L'invariant de
MacRae convoité est le pgcd de ces 3876 mineurs d'ordre 15. Une fois
que tout sera en place, l'invariant de MacRae pourra être obtenu comme
le pgcd de $n$ (ici 3) mineurs bien choisis.  Ci-dessous, la matrice
de $\Syl_d(\pXD)$ dans les bases indiquées :

{\small
$$
\NorthEastBordermatrix{
\Veti{X_{1}^{2}\,e_{1}} & \Veti{X_{1}X_{2}\,e_{1}} & \Veti{X_{1}X_{3}\,e_{1}} & \Veti{X_{2}^{2}\,e_{1}} & \Veti{X_{2}X_{3}\,e_{1}} & \Veti{X_{3}^{2}\,e_{1}} & \Veti{X_{1}^{3}\,e_{2}} & \Veti{X_{1}^{2}X_{2}\,e_{2}} & \Veti{X_{1}^{2}X_{3}\,e_{2}} & \Veti{X_{1}X_{2}^{2}\,e_{2}} & \Veti{X_{1}X_{2}X_{3}\,e_{2}} & \Veti{X_{1}X_{3}^{2}\,e_{2}} & \Veti{X_{2}^{3}\,e_{2}} & \Veti{X_{2}^{2}X_{3}\,e_{2}} & \Veti{X_{2}X_{3}^{2}\,e_{2}} & \Veti{X_{3}^{3}\,e_{2}} & \Veti{X_{1}\,e_{3}} & \Veti{X_{2}\,e_{3}} & \Veti{X_{3}\,e_{3}} & \\
p_{1} & . & . & . & . & . & . & . & . & . & . & . & . & . & . & . & . & . & . & \Heti{X_{1}^{4}} \\
. & p_{1} & . & . & . & . & p_{2} & . & . & . & . & . & . & . & . & . & . & . & . & \Heti{X_{1}^{3}X_{2}} \\
. & . & p_{1} & . & . & . & . & . & . & . & . & . & . & . & . & . & . & . & . & \Heti{X_{1}^{3}X_{3}} \\
. & . & . & p_{1} & . & . & . & p_{2} & . & . & . & . & . & . & . & . & . & . & . & \Heti{X_{1}^{2}X_{2}^{2}} \\
. & . & . & . & p_{1} & . & . & . & p_{2} & . & . & . & . & . & . & . & . & . & . & \Heti{X_{1}^{2}X_{2}X_{3}} \\
. & . & . & . & . & p_{1} & . & . & . & . & . & . & . & . & . & . & . & . & . & \Heti{X_{1}^{2}X_{3}^{2}} \\
. & . & . & . & . & . & . & . & . & p_{2} & . & . & . & . & . & . & . & . & . & \Heti{X_{1}X_{2}^{3}} \\
. & . & . & . & . & . & . & . & . & . & p_{2} & . & . & . & . & . & . & . & . & \Heti{X_{1}X_{2}^{2}X_{3}} \\
. & . & . & . & . & . & . & . & . & . & . & p_{2} & . & . & . & . & . & . & . & \Heti{X_{1}X_{2}X_{3}^{2}} \\
. & . & . & . & . & . & . & . & . & . & . & . & . & . & . & . & p_{3} & . & . & \Heti{X_{1}X_{3}^{3}} \\
. & . & . & . & . & . & . & . & . & . & . & . & p_{2} & . & . & . & . & . & . & \Heti{X_{2}^{4}} \\
. & . & . & . & . & . & . & . & . & . & . & . & . & p_{2} & . & . & . & . & . & \Heti{X_{2}^{3}X_{3}} \\
. & . & . & . & . & . & . & . & . & . & . & . & . & . & p_{2} & . & . & . & . & \Heti{X_{2}^{2}X_{3}^{2}} \\
. & . & . & . & . & . & . & . & . & . & . & . & . & . & . & p_{2} & . & p_{3} & . & \Heti{X_{2}X_{3}^{3}} \\
. & . & . & . & . & . & . & . & . & . & . & . & . & . & . & . & . & . & p_{3} & \Heti{X_{3}^{4}} \\
}
$$
}

\noindent
Quelques constats immédiats 
(les deux premiers sont valables quel que soit le degré $d$) :
\label{ConstatsImmediats}
\begin{enumerate}[{[}1{]}]
\item 
chaque colonne possède un seul coefficient non nul, qui est un certain $p_j$.
Plus précisément, la colonne $X^\gamma e_j$ possède le coefficient $p_j$ 
à la ligne $X^\gamma X_j^{d_j}$.

\item 
un coefficient $p_i$ en ligne $X^\alpha$ connaît sa colonne : c'est celle 
indexée par $(X^\alpha/X_i^{d_i})\,e_i$. Un tel $X^\alpha$ est donc 
divisible par $X_i^{d_i}$.

\item 
pour $d \geqslant \delta+1$, chaque ligne possède au moins un $p_i$.
Ceci est dû au fait qu'un monôme $X^\alpha$ de degré $d \geqslant \delta+1$ 
est au moins divisible par un $X_i^{d_i}$.

Ainsi, pour un tel $i$, la ligne $X^\alpha$ possède le coefficient $p_i$ en
la colonne $(X^\alpha/X_i^{d_i})\,e_i$.
\end{enumerate}

\bigskip

Examinons les mineurs pleins de la matrice ci-dessus.  D'après [1], un
tel mineur possède autant de coefficients non nuls (qui sont des
$p_j$) que sa taille.  Il y a donc deux types de mineurs pleins.

$\blacktriangleright$
D'abord ceux possédant deux colonnes proportionnelles $p_iX^\alpha$,
$p_jX^\alpha$ avec $i \ne j$. Ce qui implique l'existence d'une ligne
nulle. En effet, toute colonne possède un seul $p_k$, donc le nombre
de $p_k$ dans le mineur est la taille du mineur \idest{} le nombre de
lignes; puisque qu'il y a $p_i, p_j$ distincts dans une même ligne,
c'est qu'il y a une ligne sans aucun $p_k$. Réciproquement, via
un argument analogue, la présence d'une ligne nulle
implique l'existence de deux colonnes proportionnelles.
Bref, deux bonnes raisons qu'un tel mineur soit de valeur nulle.

\smallskip

Comme exemple de mineur nul, on peut prendre celui s'appuyant sur les
$15$ premières colonnes (donc en délaissant les 4 dernières colonnes):
les colonnes 2,7 sont proportionnelles ainsi que les colonnes 4,8 et
les colonnes 5,9. Et il y a trois lignes nulles (la ligne 10 et les
deux dernières).

\medskip
$\blacktriangleright$
Et les autres, ceux qui nous
intéressent, pour lesquels les $p_i$ sont disposés comme dans une matrice
de permutation (chaque ligne contient un unique~$p_i$).

\medskip
Notre objectif est d'obtenir tous les mineurs pleins
du type \og matrice de permutation \fg{}.  Comme chaque ligne
$X^\alpha$ possède un $p_i$ (point [3]), qui lui-même connaît sa
colonne (point [2]), un tel mineur spécifie une injection \og lignes
$\hookrightarrow$ colonnes \fg{} : 

$$
X^\alpha \ \longmapsto \ \dfrac{X^\alpha}{X_i^{d_i}}e_i 
\qquad 
\text{où $i \in \DivSeq(X^\alpha)
 = \big\{j \in \{1..n\} \mid X_j^{d_j} \text{ divise } X^\alpha\big\}$ 
}
$$
Réciproquement, on peut construire des injections \og lignes
$\hookrightarrow$ colonnes \fg{} 
$$
X^\alpha \ \longmapsto \ \dfrac{X^\alpha}{X_i^{d_i}}e_i 
\qquad 
\text{où $i$ est n'importe quel élément de $\DivSeq(X^\alpha)$}
$$
Le lecteur vérifiera que l'application ci-dessus est bien injective. Et s'il s'inquiète de 
sa bonne définition, rappelons-lui que $\DivSeq(X^\alpha)$ 
est non vide car $X^\alpha$ est de degré~$\geqslant \delta+1$.

\index{ensemble de divisibilité (d'un monôme)}%

Une telle injection permet de sélectionner des colonnes: on parlera de
\textit{mécanisme de sélection}.

En terme d'exposants, un mécanisme de sélection, relativement au format $D$,
est donc une application du type :
$$
\varkappa : 
\begin{array}[t]{rcl}
\bbN^n_d & \longrightarrow & \{1..n\} \\
\alpha & \longmapsto & i \qquad \text{avec la propriété $\alpha_i \geqslant d_i$}
\end{array}
$$
Un tel mécanisme de sélection, indépendamment du système $\uP$ de
format $D$, fournit la \textit{description} des colonnes d'un mineur
plein.  Et, pour tout système $\uP$, la valeur du mineur est
parfaitement définie, et pas seulement au signe près, grâce à la
bijection entre l'ensemble des lignes et l'ensemble des colonnes, peu
importe la manière dont les lignes sont rangées.

\index{mécanisme de sélection}%
\label{NOTA04-Nnd}
\label{NOTA04-kappa}

Un tel mineur (équipé d'une bijection entre l'ensemble des lignes et
l'ensemble des colonnes) est à rapprocher d'une \textit{matrice
strictement carrée}; une matrice strictement carrée n'est pas une
matrice dont le nombre de lignes est égal au nombre de colonnes mais
une matrice $(a_{ij})_{(i,j) \in I \times I}$ dont l'ensemble des indices
de lignes est \textit{égal} à l'ensemble des indices de colonnes
et de ce fait possède un déterminant, sans que l'ensemble $I$
ne soit ordonné.

\index{matrice!strictement carrée}%

\medskip
Enfin, dans le cadre jeu étalon généralisé, tout mécanisme de
sélection fournit un mineur non nul dont la valeur est le produit des
$p_k$ dans le mineur (il n'y a pas de signe).  Et dans ce cadre, on
obtient, pour $d \geqslant \delta+1$, que le nombre de mineurs non nuls est
$\prod\limits_{|\alpha| = d} \# \DivSeq(X^\alpha)$.

\bigskip
Reprenons l'exemple précédent. 
Voici un tableau \label{TableauDiv213} 
donnant,  
pour chaque monôme $X^\alpha$, 
l'ensemble $\DivSeq(X^\alpha)$ 
et en-dessous, les colonnes correspondantes :
{\small
$$
\setlength{\arraycolsep}{0.5\arraycolsep}
\begin {array}{c|c|c|c|c|c|c|c|c|c|c|c|c|c|c}
X_1^4& X_1^3X_2& X_1^3X_3& X_1^2X_2^2& X_1^2X_2X_3& X_1^2X_3^2& X_1X_2^3& X_1X_2^2X_3&
X_1X_2X_3^2& X_1X_3^3& X_2^4& X_2^3X_3& X_2^2X_3^2& X_2X_3^3& X_3^4 \\
\hline
1 & 1,2 & 1 & 1,2 &  1,2 & 1 & 2 & 2 & 2 & 3 & 2 & 2 & 2 & 2,3 & 3 \\
\hline
\Veti{X_{1}^{2}\,e_{1}} &
\Veti{X_{1}X_2\,e_{1}} \ 
\Veti{X_{1}^{3}\,e_{2}} &
\Veti{X_{1}X_{3}\,e_{1}} & 
\Veti{X_{2}^2 \,e_{1}} \ 
\Veti{X_{1}^{2}X_{2}\,e_{2}} & 
\Veti{X_{2}X_{3}\,e_{1}} \ 
\Veti{X_{1}^2 X_3 \,e_{2}} & 
\Veti{X_{3}^{2}\,e_{1}}& 
\Veti{X_{1}X_{2}^{2}\,e_{2}} & 
\Veti{X_{1}X_{2}X_{3}\,e_{2}} & 
\Veti{X_{1}X_{3}^{2}\,e_{2}} & 
\Veti{X_{1}\,e_{3}} & 
\Veti{X_{2}^{3}\,e_{2}} & 
\Veti{X_{2}^{2}X_{3}\,e_{2}} & 
\Veti{X_{2}X_{3}^{2}\,e_{2}} & 
\Veti{X_{3}^{3}\,e_{2}} \ 
\Veti{X_{2}\,e_{3}} & 
\Veti{X_{3}\,e_{3}} 
\end{array}
$$
}
\label{TableauExemple213-d4}
On constate qu'il y a 
$1^{11} \times 2^4 = 16$ mineurs non nuls parmi les $3876$.

\medskip
Nous venons ainsi de mettre en lumière $\DivSeq(X^\alpha)$,
ensemble de divisibilité de $X^\alpha$ relativement au format $D$, qui
joue un rôle capital dans notre combinatoire de l'élimination.

\begin{prop} [Où le scalaire $p_1^{\widehat d_1}\cdots p_n^{\widehat d_n}$ fait son apparition]
  \label{MacRaeJeuEtalon} 
Soit $d\geqslant \delta+1$.
\leavevmode
\begin{enumerate}[\rm i)]
\item 
Pour $i$ fixé, le nombre de monômes $X^\alpha$ 
de degré $d$ tels que $\DivSeq(X^\alpha) = \{i\}$ 
est $\widehat d_i = \prod_{j\neq i} d_j$.

\item 
Les mineurs pleins de $\Syl_d(\pXD)$ sont ou bien nuls ou bien, au signe près, 
des monômes en $p_1, \dots, p_n$.
Il n'y a donc pas d'intervention du petit anneau de base $\bfk$.

\item
De manière plus précise, chaque mineur plein de $\Syl_d(\pXD)$ non nul
définit un mécanisme de sélection et la valeur du mineur est du type
$p_1^{\nu_1} \cdots p_n^{\nu_n}$ avec $\nu_i \geqslant \widehat d_i$
pour tout $i$.

De plus, pour $i$ fixé, il existe un mineur explicite vérifiant $\nu_i = \widehat d_i$.

\item 
Le pgcd fort des mineurs pleins de $\Syl_d(\pXD)$ est 
$p_1^{\widehat d_1}\cdots p_n^{\widehat d_n}$. 

De plus, il y a $n$ mineurs explicites ayant pour pgcd fort 
$p_1^{\widehat d_1}\cdots p_n^{\widehat d_n}$. 

\item 
Le $\bfA$-module $\bfB_d$ est de MacRae de rang $0$ et son invariant de MacRae
$\MacRae(\bfB_d)$ est l'idéal principal de $\bfA$ engendré par le scalaire régulier 
$p_1^{\widehat d_1}\cdots p_n^{\widehat d_n}$.
\end{enumerate}
\end{prop}

\medskip

Rappelons que nous avons donné en~\ref{DefIdealModuleExcedentaire}
le nom $\Jid$  au $\bfA$-module engendré 
par les monômes $X^\alpha$ de degré~$d$ tels que $\DivSeq(X^\alpha) = \{i\}$ 
qui interviennent dans le point i). Par ailleurs, pour $d \geqslant \delta+1$,
un monôme de degré $d$ est divisible par au moins un $X_i^{d_i}$, ce
qui conduit à
$$
\bfA[\uX]_d \ = \ \Jex_{1,d} \ = \
\Jex_{1 \setminus 2, d} 
\, \oplus \, 
\Jex_{2, d}
\ = \ \bigoplus_{i=1}^{n} \Jid
\, \oplus \, 
\Jex_{2, d}
$$
Dans la preuve ci-dessous, pour $i$ fixé, nous écrirons cette égalité sous la forme:
$$
\bfA[\uX]_d \ = \ 
\Jid \ \oplus \ 
\Big(\bigoplus_{j \neq i} \Jex_{1\setminus 2,d}^{(j)} \oplus \Jex_{2, d} \Big)
$$

\begin{proof} [Preuve de la proposition]
\leavevmode

\noindent
i) Confer~\ref{LemmeDenombrementJ1moins2id}.

\smallskip

\noindent
ii) Confer les constats immédiats et l'exemple avant la proposition.

\smallskip

\noindent
iii) 
Comme déjà dit au début 
de~\ref{soussectionJeuEtalonGenDeltaPlus1},
un mineur plein non nul se construit via un mécanisme de sélection:
$$
X^\alpha \ \longmapsto \ \dfrac{X^\alpha}{X_{i_\alpha}^{d_{i_\alpha}}}e_{i_\alpha} 
\qquad 
\text{où $i_\alpha$ est n'importe quel élément de $\DivSeq(X^\alpha)$}
$$
La valeur d'un tel mineur est $p_1^{\nu_1} \cdots p_n^{\nu_n}$ 
où l'exposant $\nu_j$ est égal au nombre de $X^\alpha$ tels que $i_\alpha = j$.

Fixons $i$ et évaluons le nombre de $X^\alpha$ pour lesquels la sélection 
a fourni $i_\alpha = i$. 
Commençons par regarder ce qu'a donné la sélection au niveau des $X^\alpha \in \Jid$
pour lesquels $\DivSeq(X^\alpha)$ est le singleton~$\{i\}$.
La réponse est simple et est presque tautologique : pour chaque $X^\alpha \in \Jid$, 
la sélection fournit nécessairement $i_\alpha = i$ (seul $i$ peut être sélectionné, puisqu'il 
est tout seul !).
On vient de montrer que $\nu_i$ est nécessairement au moins égal à $\dim \Jid$.
D'après i), on obtient donc $\nu_i \geqslant \widehat d_i$.

\bigskip

Voyons comment construire un mineur tel que $\nu_i = \widehat d_i$.
Il s'agit d'adapter ce que l'on vient de dire. Voilà une sélection de colonnes 
qui fournit un mineur comme annoncé :
$$
X^\alpha \ \longmapsto \ 
\left\{
\begin{array}{ll}
\dfrac{X^\alpha}{X_i^{d_i}} e_i & \text{si $X^\alpha \in \Jid$ (pas le choix)}
\\ [0.4cm]
\dfrac{X^\alpha}{X_j^{d_j}} e_j & \text{sinon, où $j$ est quelconque 
dans $\DivSeq(X^\alpha)$ et distinct de $i$}
\end{array}
\right.
$$

iv) 
Posons $\calR = p_1^{\widehat d_1}\cdots p_n^{\widehat d_n}$.  D'après
iii), cet élément régulier $\calR$ divise tous les mineurs pleins de
$\Syl_d$.  D'après le point précédent, pour tout $i$, il existe un
mineur plein s'écrivant $\calR \times \fm_i$ où $\fm_i$ est un monôme
en les $p_j$ ne dépendant pas de $p_i$.  La suite $(\fm_1, \dots,
\fm_n)$ est de profondeur $\geqslant 2$ (cela résulte du fait que, sur
un anneau quelconque, une famille finie de monômes premiers entre eux
est de profondeur $\geqslant 2$, cf.~\ref{MonomesPremiersEntreEux}).
On en déduit que le pgcd de ces $n$ mineurs pleins est exactement
$\calR$.

Enfin, comme $\calR$ divise \textit{tous} les mineurs pleins, le pgcd fort de tous les 
mineurs pleins est $\calR$.

\medskip
v) 
D'après iv), l'idéal $\calD_{s_d}(\Syl_d)$ admet pour pgcd fort $\calR = 
p_1^{\widehat d_1}\cdots p_n^{\widehat d_n}$.
D'après la définition~\ref{DefMacRae0}, 
on en déduit que $\bfB_d$ est un module de MacRae de rang $0$ 
d'invariant de MacRae $\calR$. 
\end{proof}

\subsubsection*{Retour sur l'exemple $D = (2,1,3)$ et $d = \delta+1 = 4$, page~\pageref{Exemple213-d4}}

Dans cet exemple, nous allons cerner les mineurs ayant l'exposant
minimal $\widehat d_3 = d_1d_2 \buildrel {\rm ici} \over= 2$ en $p_3$ puis ceux ayant l'exposant
minimal $\widehat d_2 = d_1d_3 \buildrel {\rm ici} \over= 6$ en $p_2$.

\medskip
\noindent
$\blacktriangleright$
Exposant minimal en $p_3$, ici $p_3^2$.
Afin d'obtenir un tel mineur, il faut commencer par examiner les monômes
$X^\alpha$ pour lesquels $\DivSeq(X^\alpha)$ contient strictement $3$
(et donc contient un autre entier, au moins).  
Pour une telle ligne
$X^\alpha$, il faut sélectionner un indice autre que $3$.  
Le tableau page~\pageref{TableauExemple213-d4}
montre qu'il n'y a qu'une seule ligne de ce type, à savoir celle
indexée par $X_2 X_3^3$, monôme ayant comme ensemble de divisibilité
$\{2,3\}$.  Pour ce monôme, il faut donc sélectionner $2$,
c'est-à-dire la colonne $X_3^3 e_2$ (et non la colonne $X_2 e_3$).
Cette contrainte étant respectée, la sélection peut être totalement
quelconque pour les autres lignes.  Ici, il y a donc $2^3 = 8$ mineurs
ayant un exposant minimal en $p_3$.

\smallskip
Parmi ces $8$ mineurs, il y en a un bien précis, obtenu à partir du
mécanisme de sélection $i_\alpha = \min\big(\DivSeq(X^\alpha)\big)$.
Si l'on décide de ranger les colonnes relativement à l'injection de
sélection, c'est le déterminant de la matrice diagonale :

{\small
$$
\EastBordermatrix{
p_{1} & . & . & . & . & . & . & . & . & . & . & . & . & . & . & \Heti{X_{1}^{4}} \\ 
. & p_{1} & . & . & . & . & . & . & . & . & . & . & . & . & . & \Heti{X_{1}^{3}X_{2}} \\ 
. & . & p_{1} & . & . & . & . & . & . & . & . & . & . & . & . & \Heti{X_{1}^{3}X_{3}} \\ 
. & . & . & p_{1} & . & . & . & . & . & . & . & . & . & . & . & \Heti{X_{1}^{2}X_{2}^{2}} \\ 
. & . & . & . & p_{1} & . & . & . & . & . & . & . & . & . & . & \Heti{X_{1}^{2}X_{2}X_{3}} \\ 
. & . & . & . & . & p_{1} & . & . & . & . & . & . & . & . & . & \Heti{X_{1}^{2}X_{3}^{2}} \\ 
. & . & . & . & . & . & p_{2} & . & . & . & . & . & . & . & . & \Heti{X_{1}X_{2}^{3}} \\ 
. & . & . & . & . & . & . & p_{2} & . & . & . & . & . & . & . & \Heti{X_{1}X_{2}^{2}X_{3}} \\ 
. & . & . & . & . & . & . & . & p_{2} & . & . & . & . & . & . & \Heti{X_{1}X_{2}X_{3}^{2}} \\ 
. & . & . & . & . & . & . & . & . & p_{3} & . & . & . & . & . & \Heti{X_{1}X_{3}^{3}} \\ 
. & . & . & . & . & . & . & . & . & . & p_{2} & . & . & . & . & \Heti{X_{2}^{4}} \\ 
. & . & . & . & . & . & . & . & . & . & . & p_{2} & . & . & . & \Heti{X_{2}^{3}X_{3}} \\ 
. & . & . & . & . & . & . & . & . & . & . & . & p_{2} & . & . & \Heti{X_{2}^{2}X_{3}^{2}} \\ 
. & . & . & . & . & . & . & . & . & . & . & . & . & p_{2} & . & \Heti{X_{2}X_{3}^{3}} \\ 
. & . & . & . & . & . & . & . & . & . & . & . & . & . & p_{3} & \Heti{X_{3}^{4}} \\ 
}
$$
}
\label{pageAllusionW1dpXD}
C'est exactement l'endomorphisme $W_{1,d}(\pXD)$ qui sera décrit en~\ref{DefW1},
ici de déterminant $p_1^6p_2^7p_3^2$.

\medskip
\noindent 
$\blacktriangleright$
Exposant minimal en $p_2$, ici $p_2^6$.
En examinant le tableau des indices de divisibilité page \pageref{TableauDiv213}, 
on constate que les $X^\alpha$ se répartissent en deux catégories disjointes:
$$
\begin{array}{l}
\text{$\DivSeq(X^\alpha)$ contient strictement $\{2\}$} \\ 
\text{\idest{} $\DivSeq(X^\alpha) = \{2,i_\alpha\}$ avec $i_\alpha=1$ ou $3$} 
\end{array}
\hspace{2cm}
\text{ou bien} \hspace{2cm}
\begin{array}{l}
\text{$\DivSeq(X^\alpha) = \{i_\alpha\}$} 
\end{array}
$$
Ces deux circonstances (à gauche, à droite), très particulières, font
qu'un tel mineur d'exposant minimal~$\widehat d_2$ en $p_2$ est unique
: dans le cas des $X^\alpha$ de la catégorie droite, il n'y a pas le
choix dans la sélection, dans le cas des $X^\alpha$ de la catégorie
gauche, il faut obligatoirement sélectionner $i_\alpha \neq 2$ pour
avoir un exposant égal à~$\widehat d_2$ en $p_2$.  Le mineur obtenu
est $p_1^6p_2^6p_3^3$.

\medskip
Il y a une manière de décrire une telle sélection (qui dans cet exemple est unique).
On constate que $i_\alpha$ se retrouve être le plus petit indice de $\DivSeq(X^\alpha)$ mais 
à condition d'avoir muni l'ensemble 
$\{1,2, \dots, n\} =_{\rm ici} \{1,2,3\}$ de l'une des deux structures
d'ordre total \textit{pour lequel 2 est le plus grand élément}:
$$
1 < 3 < 2    \qquad\hbox {ou bien}\qquad   3 < 1 < 2
\leqno (\star)
$$

\subsubsection*{Le mécanisme de sélection $\minDiv$, le mineur $\det W_{1,d}$,
                le cofacteur $b_d$ et leur $\sigma$-version}

\index{mécanisme de sélection!$\minDiv$}%
\index{mécanisme de sélection!$\sminDiv$}%

Lorsque l'on prend comme mécanisme de sélection 
celui qui consiste à choisir le plus petit indice
de divisibilité, le mineur de $\Syl_d(\uP)$ prend le nom de $\det
W_{1,d}(\uP)$ (cf. le chapitre suivant en~\ref{DefW1}).  Ainsi, dans
le tableau page~\pageref{TableauDiv213} des indices de divisibilité de
$D = (2,1,3)$, en ayant sélectionné dans les colonnes où il y a un choix
(colonnes 2,4,5,14) les indices de divisibilité minimum $1,1,1,2$, on
a obtenu $\det W_{1,d}(\pXD) = p_1^6p_2^7p_3^2$.  Plus généralement,
en notant
$$
\minDiv(X^\alpha) = \min\big(\DivSeq(X^\alpha)\big)  \qquad
\hbox {pour $X^\alpha \in \Jex_1$}
$$
on a 
$$
\det W_{1,d}(\pXD) =  \prod_{X^\alpha \in \Jex_{1,d}} p_{\minDiv(X^\alpha)}
$$
On peut également, pour toute permutation $\sigma \in \mathfrak S_n$,
considérer le mécanisme de sélection défini par le plus petit indice
de divisibilité au sens de la relation d'ordre $<_\sigma$ défini par
$\sigma(1) <_\sigma \cdots <_\sigma \sigma(n)$.
L'analogue de $\minDiv$ est noté $\sminDiv$, ce qui conduit  
à un autre mineur désigné par $\det W_{1,d}^{\sigma}(\uP)$.
Pour plus de détails, cf.~\ref{SigmaVersion}.

Par exemple, dans $(\star)$, il s'agit de l'ordre $<_\sigma$ avec, à
gauche la transposition $\sigma = (2,3)$ et à droite le cycle $\sigma
= (1,3,2)$ qui conduisent au même mineur $\det W_{1,d}^\sigma(\pXD) =
p_1^6p_2^6p_3^3$. Plus généralement:
$$
\det W_{1,d}^\sigma(\pXD) =  \prod_{X^\alpha \in \Jex_{1,d}} p_{\sminDiv(X^\alpha)}
$$

\label{NOTA04-ordreTordu}%
\index{ordre sur $\{1..n\}$ tordu par une permutation}%
\label{NOTA04-minDiv}%
\label{NOTA04-sminDiv}%
\label{NOTA04-detW1d}%
\label{NOTA04-detW1dsigma}%

Dans la proposition suivante, nous notons $\calR_d$ l'élément
$p_1^{\widehat d_1}\cdots p_n^{\widehat d_n}$ bien qu'il ne dépende
pas de $d$, ceci pour insister sur le fait qu'il intervient en chaque
degré $d \geqslant \delta+1$.
De plus, pour alléger la notation $\det W_{1,d}(\pXD)$, 
nous omettrons $\pXD$ ci-dessous (idem pour $\det W_{1,d}^\sigma$).

\label{NOTA04-Rd}%

\medskip

\begin{prop} [Où $\Jex_{2,d}$ apparaît dans le cofacteur de $\calR_d$ dans $\det W_{1,d}$]
\label{dCofacteurJeuEtalonGen}
Pour le jeu étalon généralisé $\pXD = (p_1X_1^{d_1}, \dots, p_nX_n^{d_n})$,
en degré $d \geqslant \delta + 1$:

\begin{enumerate}[\rm i)]
\item
La relation de divisibilité $\calR_d \mid \det W_{1,d}$ s'explicite
par le cofacteur $b_d$, monôme en $p_1,
\dots, p_n$ ne dépendant pas de $p_n$ :
$$
\det W_{1,d} =  b_d \ \calR_d
\qquad \text{avec }
b_d \ =\ 
\prod_{X^\alpha \in \Jex_{2,d}} p_{\minDiv(X^\alpha)} 
\ = \ 
\prod_{i=1}^n p_i^{\nu_i}
\qquad 
\text{ où $\nu_i = \dim \Jex_{2,d}^{(i)}$ ($\nu_n = 0$).}
$$

\item De manière analogue, pour une permutation $\sigma \in \mathfrak S_n$,
on dispose d'un monôme $b^\sigma_d$ en $p_1, \dots, p_n$ ne
dépendant pas de $p_{\sigma(n)}$:
$$
\det W_{1,d}^\sigma =  b_d^\sigma \ \calR_d
\qquad \text{avec }
b_d^\sigma = \prod_{X^\alpha \in \Jex_{2,d}} p_{\sminDiv(X^\alpha)} 
$$

\item 
Soit $\Sigma_2$ l'ensemble des $n$ transpositions $(i,n)$, 
$1 \leqslant i \leqslant n$.
Alors $\calR_d$ est le pgcd fort des $(\det W_{1,d}^{\sigma})_{\sigma \in \Sigma_2}$.
\end{enumerate}
\end {prop}

\label{NOTA04-bd}%
\label{NOTA04-bdsigma}%
%
%

\begin{proof}\leavevmode

\noindent  
i)  
On a $\calR_d = p_1^{\widehat d_1}\cdots p_n^{\widehat d_n}$ et
$\widehat d_i = \dim\Jex_{1\setminus2,d}^{(i)}$.
Comme $\Jex_{1 \setminus 2,d} = \bigoplus_{i=1}^n \Jex_{1\setminus2,d}^{(i)}$:
$$
\calR_d
\ = \ 
\prod_{X^\alpha \in \Jex_{1\setminus 2,d}} p_{\minDiv(X^\alpha)} 
$$
Par ailleurs, par construction:
$$
\det W_{1,d}
\ = \ 
\prod_{X^\alpha \in \Jex_{1,d}} p_{\minDiv(X^\alpha)} 
$$
Comme $\Jex_{1,d} = \Jex_{1 \setminus 2,d} \oplus \Jex_{2,d}$ 
et $\Jex_{2,d} = \bigoplus\limits_{i=1}^n \Jex_{2,d}^{(i)}$,
on obtient les deux expressions pour le multiplicateur $b_d$:
$$
b_d
\ = \ 
\prod_{X^\alpha \in \Jex_{2,d}} p_{\minDiv(X^\alpha)} 
\ = \ 
\prod_{i=1}^n p_i^{\nu_i}
\qquad 
\text{ où $\nu_i = \dim \Jex_{2,d}^{(i)}$}
$$
Le fait que $b_d$ ne dépende pas de $p_n$ provient du fait
que tout monôme $X^\alpha$ de $\Jex_2$ ayant au moins deux indices de
divisibilité est tel que $\minDiv(X^\alpha) < n$, 
ce qui se traduit par $\Jex_{2,d}^{(n)} = 0$.

\medskip
\noindent
ii) Se traite comme i).

\medskip
\noindent
iii) Comme $b_d^\sigma$ est un monôme en $p_1, \dots, p_n$ ne dépendant pas de $p_{\sigma(n)}$, 
la famille $(b_d^\sigma)_{\sigma \in \Sigma_2}$ est de profondeur $\geqslant 2$.
\end{proof}

\begin{rmq}
Les mécanismes de sélection $\minDiv$ et/ou $\sminDiv$, qui
conduisent aux mineurs notés $\det W_{1,d}^\sigma$, sont très
particuliers car composés sous la forme
$$
\xymatrix @C=1.5cm{
\{\hbox {monômes de $\Jex_1$}\} \ar[r]_-{\DivSeq}  &
\Big\{\hbox {parties non vides de $\{1,2, \dots, n\}$}\Big\}
\ar[rr]^-{\text{sélection d'un}}_-{\text{élément dans la partie}}
&& \{1,2, \dots, n\}
}
$$
En conséquence, deux monômes $X^\alpha$, $X^\beta$ ayant même 
ensemble de divisibilité, ont même image par le mécanisme de sélection, alors 
qu'avec un mécanisme quelconque
$\{\hbox {monômes de $\Jex_1$}\} \to \{1,2, \dots, n\}$, cela ne
serait pas nécessairement le cas.
\end{rmq}

\subsection{Le cas $d = \delta$. Forme déterminantale associée à une application linéaire}

\label{soussection-d-egal-delta}

Nous avons besoin ici de définir la notion de forme linéaire
déterminantale associée à une application linéaire $u : E \to F$ entre
deux modules libres de rang fini; ou encore associée à une matrice si
nous convenons de baser les modules. Nous y reviendrons dans la
section~\ref{SectionDetCoRang1}. 

\index{forme linéaire!déterminantale}%

\subsubsection*{Forme linéaire déterminantale pure}

Désignons par $\ell = \dim F$ (le nombre de lignes de la
matrice). Tout choix de $\ell-1$ colonnes fournit une forme
linéaire~$\mu$ sur l'espace d'arrivée $F$ de $u$. Par exemple, en
présence de bases ordonnées $\bfe$ de $E$ et $\bff$ de $F$, la forme
linéaire déterminantale $\mu : F \to \bfA$ définie par le choix des
$\ell-1$ premières colonnes de $u$ est par définition:
$$
\mu : \ v \ \longmapsto \ 
\det\nolimits_\bff(u(e_1), \dots, u(e_{\ell-1}), v)  
\leqno (\star)
$$
On peut également choisir $\ell -1$ colonnes quelconques.
Dans l'exemple ci-dessous, $u$ est à valeurs dans $\bfA^3 = \bfA
f_1 \oplus \bfA f_2 \oplus \bfA f_3$, et le choix porte sur les deux
colonnes $1,4$:
$$
\begin {bmatrix}
a_1 & a_2 & a_3 & a_4 & a_5 \\
b_1 & b_2 & b_3 & b_4 & b_5 \\
c_1 & c_2 & c_3 & c_4 & c_5 \\
\end {bmatrix}
\qquad\qquad
\mu_{1,4} : (x,y,z) \ \longmapsto \ 
\begin {vmatrix}
a_1  &  a_4  & x \\
b_1  &  b_4  & y \\
c_1  &  c_4  & z \\
\end {vmatrix}
\quad \text{ou encore}\quad
\mu_{1,4} = 
\begin {vmatrix}
a_1  &  a_4  & f_1^\star \\
b_1  &  b_4  & f_2^\star \\
c_1  &  c_4  & f_3^\star \\
\end {vmatrix}
$$
Il y a aussi 
$$
\mu_{4,1} : (x,y,z) 
\ \longmapsto \
\begin {vmatrix}
a_4  &  a_1  & x \\
b_4  &  b_1  & y \\
c_4  &  c_1  & z \\
\end {vmatrix}
$$
De telles formes linéaires seront appelées formes déterminantales
pures.  Le sous-module de $F^\star$ engendré par ces formes est noté
$\DVect_{\ell-1}(u)$: c'est le sous-module des formes déterminantales de $u$.
Il est indépendant des bases et de la position du vecteur $v$ dans le déterminant $(\star)$.
En remplaçant dans $(\star)$ les $u(e_j)$ par~$\ell-1$ vecteurs
quelconques de $\Im u$, on obtient une forme linéaire qui est dans
$\DVect_{\ell-1}(u)$.

\index{forme linéaire!déterminantale pure}
\label{NOTA04-Dvect}
%
%

\subsubsection*{Forme déterminantale pure sans ambiguïté de signe
                via le procédé $(\calB_F, f, \iota)$}

Supposons $F$ muni d'une base $\calB_F$ que nous ne
supposons pas ordonnée. Cette précision est importante car la plupart
des modules libres qui interviennent dans notre étude du résultant
possèdent cette propriété. Ainsi $\bfA[\uX]_d$ dispose de sa base
monomiale non ordonnée et plus généralement, chaque terme $\rmK_{k,d}$
de la composante homogène de degré $d$ du complexe de Koszul de $\uP$
est équipé de sa base monomiale.

Fixons dans la base $\calB_F$ un vecteur $f$ que nous qualifions de distingué.
Alors toute application $\iota : \calB_F \setminus \{f\} \to E$ permet de spécifier sans
ambiguité de signe une forme linéaire pure associée à une 
application linéaire $u : E \to F$ et sous entendu à $(\calB_F, f, \iota)$.

Pour la décrire, prenons une numérotation quelconque $f_1, \dots, f_\ell$ de $\calB_F$,
désignons par~$\bff$ la base \textit{ordonnée} $(f_1, \dots,
f_\ell)$ et par $[\qquad]_\bff$ le déterminant de $\ell$ vecteurs de $F$
dans cette base. 
Si nous supposons le vecteur distingué~$f$ en dernière
position par exemple, la forme déterminantale pure associée est:
$$
v \ \longmapsto \ 
\big[u\big(\iota(f_1)\big), 
\dots, 
u\big(\iota(f_{\ell-1})\big), v\big]_\bff  \qquad v \in F
$$
Il faut bien sûr adapter l'expression déterminantale selon la position
de $f$ et mettre en adéquation le vecteur $f$ et la colonne-argument.
Ainsi, si nous supposons $f$ en position $2$, la
forme linéaire est
$$
v \ \longmapsto \ 
\big[ u\big(\iota(f_1)\big), \, v,\,  
u\big(\iota(f_3)\big), \dots, u\big(\iota(f_\ell)\big)\big]_\bff   
$$
On obtient ainsi une forme linéaire indépendante de la numérotation.

\medskip

Reprenons l'exemple précédent en supposant le vecteur privilégié 
$f$ en première position (donc $f = f_1$) et 
$\iota(f_2) = e_4$,  $\iota(f_3) = e_1$.
La forme linéaire associée est :
$$
(x,y,z) 
\ \longmapsto\ 
\begin {vmatrix}
x  &  a_4  & a_1 \\
y  &  b_4  & b_1 \\
z  &  c_4  & c_1 \\
\end {vmatrix}
$$
Cette forme linéaire est bien sûr indépendante 
du rangement des vecteurs de base. 
Par exemple, relativement aux rangements $(f_2, f_1, f_3)$ 
et $(f_3, f_1, f_2)$,  
la forme linéaire a pour expression :
$$
(x,y,z) 
\ \longmapsto\ 
\begin {vmatrix}
b_4  &  y  & b_1 \\
a_4  &  x  & a_1 \\
c_4  &  z  & c_1 \\
\end {vmatrix}
\ = \ 
\begin {vmatrix}
c_1  &  z  & c_4 \\
a_1  &  x  & a_4 \\
b_1  &  y  & b_4 \\
\end {vmatrix}
$$

\subsubsection*{Le cas de $\Syl_\delta(\uP) : \rmK_{1,\delta} \to \bfA[\uX]_\delta$:
                le sous-module $\DVect_{s_\delta}(\Syl_\delta) \subset \bfA[\uX]_\delta^\star$}

Nous allons maintenant appliquer ce qui précède aux $\bfA$-modules $E = \rmK_{1,\delta}$
et $F = \bfA[\uX]_\delta$ équipés de leur base monomiale, en prenant
comme vecteur privilégié $f$ le \MoutonNoir{} $X^\emouton$ et comme application
linéaire l'application de Sylvester
$$
\Syl_\delta(\uP) : 
\rmK_{1,\delta} = \bigoplus_{i=1}^n  \bfA[\uX]_{\delta-d_i} e_i 
\ \longrightarrow \ 
\bfA[\uX]_\delta
\qquad\qquad
U_1 e_1 + \cdots + U_n e_n  
\ \longmapsto \ 
U_1P_1 + \cdots + U_nP_n
$$
En notant $\calB$ la base monomiale de $\bfA[\uX]_\delta$, on a donc
$\calB \setminus \{X^\emouton\} = \{ \text{monômes de } \Jex_{1,\delta}\}$.
Quant aux applications $\iota : \calB\setminus\{X^\emouton\}
\to \rmK_{1,\delta}$, ce sont celles associées à un mécanisme de
sélection $\kappa$ :
$$
\iota : 
\begin{array}[t]{rcl}
\{\hbox {monômes de $\Jex_{1,\delta}$}\} & \longrightarrow & \rmK_{1,\delta} \\ [0.2cm]
X^\alpha & \longmapsto & \dfrac{X^\alpha}{X_i^{d_i}}\, e_i  
\hbox { avec $i = 
\kappa(\alpha)$}
\end{array}
$$
Nous nous limitons ici à $\uP = \pXD$ mais il est bon d'avoir à
l'esprit que ces constructions ne dépendent que du format $D$ et
permettront plus tard d'élaborer des objets s'appliquant à n'importe
quel système $\uP$ de format $D$ (cf. le
chapitre~\ref{ChapObjetsSylvester}).  En quelque sorte, le jeu
étalon généralisé sert de charpente pour tous les systèmes~$\uP$
de format~$D$.

\index{mécanisme de sélection}%

\medskip
Dans l'énoncé ci-dessous, nous rappelons que nous avons noté $(X^{\emouton})^\star$
la forme linéaire coordonnée sur le \MoutonNoir{} relativement
à la base monomiale de $\bfA[\uX]_\delta$.

\begin{prop} 
[Où la forme linéaire $p_1^{\widehat d_1-1}\cdots p_n^{\widehat d_n -1}(X^\emouton)^\star$ fait son apparition] 
\label{MacRaeJeuEtalonDelta} \leavevmode
  
\begin{enumerate}[\rm i)]
\item 
Pour $i$ fixé, le nombre de monômes $X^\alpha$ 
de degré $\delta$ tels que $\DivSeq(X^\alpha) = \{i\}$ 
est $\widehat d_i - 1 = \prod_{j\neq i} d_j -1$.

\item 
Les mineurs d'ordre $s_\delta+1 = \dim \bfA[\uX]_\delta$ de $\Syl_\delta(\pXD)$ sont nuls.

Toute forme déterminantale pure de $\Syl_\delta(\pXD)$ est ou bien nulle ou bien, au signe près, 
un multiple de la forme linéaire $(X^{\emouton})^\star$, le multiplicateur étant 
un monôme en $p_1, \dots, p_n$. 
De manière précise, chaque forme déterminantale pure non nulle de
$\Syl_\delta(\pXD)$ définit un mécanisme de sélection $\kappa$.

\item
Pour une forme déterminantale pure $\mu$ associée à un mécanisme de
sélection, on a une égalité sans signe $\mu = p_1^{\nu_1} \cdots
p_n^{\nu_n} (X^\emouton)^\star$ où $\nu_i$ vérifie $\nu_i \geqslant
\widehat d_i -1 $ pour tout~$i$.

De plus, pour chaque $1\le i\le n$, il existe une forme déterminantale
pure explicite vérifiant $\nu_i = \widehat d_i - 1$.

\item 
Posons $\omegares = p_1^{\widehat d_1-1}\cdots p_n^{\widehat d_n -1} (X^\emouton)^\star$.

Cette forme linéaire $\omegares$ est le pgcd fort du sous-module
$\DVect_{s_\delta}(\Syl_\delta(\pXD)) \subset \bfA[\uX]_\delta^*$.

De plus, il y a $n$ formes déterminantales pures explicites ayant pour pgcd fort $\omegares$. 

\item
Le $\bfA$-module $\bfB_\delta$ est de MacRae de rang $1$ et
son invariant de MacRae $\MacRaeVect(\bfB_\delta)$ est le sous-module de $\bfB_\delta^\star$
engendré par la forme linéaire $\omegares$ passée au quotient sur $\bfB_\delta$.
\end{enumerate}
\end{prop}

\label{NOTA04-DvectSyldelta}
%
%

Avant de donner la preuve, notons que les deux premiers constats
observés précédemment page~\pageref{ConstatsImmediats} restent
valables.  Le troisième est modifié comme suit :

\begin{enumerate}[{[}1{]}]
\setcounter{enumi}{2}
\item 
comme $d = \delta$, il y a une seule ligne qui ne possède pas de $p_i$ et qui est donc nulle, 
c'est celle indexée par le \MoutonNoir. 
Les autres lignes sont indexées par des monômes $X^\alpha$ divisibles
par au moins un~$X_i^{d_i}$ et contiennent donc au moins un $p_i$.
\end{enumerate}

\noindent
Illustrons cela avec le format $D=(2,4,1)$ pour lequel $\delta = 4$ et $s_\delta = 15-1$
(15 est le nombre de monômes de degré 4 en 3 variables).

{\small
$$
\Syl_\delta(\pXD)
\ = \ 
\NorthEastBordermatrix{
\Veti{X_{1}^{2}\,e_{1}} & \Veti{X_{1}X_{2}\,e_{1}} & \Veti{X_{1}X_{3}\,e_{1}} & \Veti{X_{2}^{2}\,e_{1}} & \Veti{X_{2}X_{3}\,e_{1}} & \Veti{X_{3}^{2}\,e_{1}} & \Veti{e_{2}} & \Veti{X_{1}^{3}\,e_{3}} & \Veti{X_{1}^{2}X_{2}\,e_{3}} & \Veti{X_{1}^{2}X_{3}\,e_{3}} & \Veti{X_{1}X_{2}^{2}\,e_{3}} & \Veti{X_{1}X_{2}X_{3}\,e_{3}} & \Veti{X_{1}X_{3}^{2}\,e_{3}} & \Veti{X_{2}^{3}\,e_{3}} & \Veti{X_{2}^{2}X_{3}\,e_{3}} & \Veti{X_{2}X_{3}^{2}\,e_{3}} & \Veti{X_{3}^{3}\,e_{3}} & \\
p_{1} & . & . & . & . & . & . & . & . & . & . & . & . & . & . & . & . & \Heti{X_{1}^{4}} \\
. & p_{1} & . & . & . & . & . & . & . & . & . & . & . & . & . & . & . & \Heti{X_{1}^{3}X_{2}} \\
. & . & p_{1} & . & . & . & . & p_{3} & . & . & . & . & . & . & . & . & . & \Heti{X_{1}^{3}X_{3}} \\
. & . & . & p_{1} & . & . & . & . & . & . & . & . & . & . & . & . & . & \Heti{X_{1}^{2}X_{2}^{2}} \\
. & . & . & . & p_{1} & . & . & . & p_{3} & . & . & . & . & . & . & . & . & \Heti{X_{1}^{2}X_{2}X_{3}} \\
. & . & . & . & . & p_{1} & . & . & . & p_{3} & . & . & . & . & . & . & . & \Heti{X_{1}^{2}X_{3}^{2}} \\
. & . & . & . & . & . & . & . & . & . & . & . & . & . & . & . & . & \Heti{X_{1}X_{2}^{3}} \mouton \\
. & . & . & . & . & . & . & . & . & . & p_{3} & . & . & . & . & . & . & \Heti{X_{1}X_{2}^{2}X_{3}} \\
. & . & . & . & . & . & . & . & . & . & . & p_{3} & . & . & . & . & . & \Heti{X_{1}X_{2}X_{3}^{2}} \\
. & . & . & . & . & . & . & . & . & . & . & . & p_{3} & . & . & . & . & \Heti{X_{1}X_{3}^{3}} \\
. & . & . & . & . & . & p_{2} & . & . & . & . & . & . & . & . & . & . & \Heti{X_{2}^{4}} \\
. & . & . & . & . & . & . & . & . & . & . & . & . & p_{3} & . & . & . & \Heti{X_{2}^{3}X_{3}} \\
. & . & . & . & . & . & . & . & . & . & . & . & . & . & p_{3} & . & . & \Heti{X_{2}^{2}X_{3}^{2}} \\
. & . & . & . & . & . & . & . & . & . & . & . & . & . & . & p_{3} & . & \Heti{X_{2}X_{3}^{3}} \\
. & . & . & . & . & . & . & . & . & . & . & . & . & . & . & . & p_{3} & \Heti{X_{3}^{4}} \\
}
$$
}
\label{TableauDiv241}
Voici le tableau des ensembles de divisibilité de $D = (2,4,1)$ en degré $\delta$ :
{\small
$$
\setlength{\arraycolsep}{0.5\arraycolsep}
\begin {array}{c|c|c|c|c|c|c|c|c|c|c|c|c|c|c}
X_1^4& X_1^3X_2& X_1^3X_3& X_1^2X_2^2& X_1^2X_2X_3& X_1^2X_3^2& X_1X_2^3& X_1X_2^2X_3&
X_1X_2X_3^2& X_1X_3^3& X_2^4& X_2^3X_3& X_2^2X_3^2& X_2X_3^3& X_3^4 \\
\hline
1 & 1 & 1,3 & 1 & 1,3 & 1,3 & \mouton & 3 & 3 & 3 & 2 & 3 & 3 & 3 & 3 \\
\hline
\Veti{X_{1}^{2}\,e_{1}} &
\Veti{X_{1}X_2\,e_{1}} &
\Veti{X_{1}X_{3}\,e_{1}} \ 
\Veti{X_1^3 \, e_3} & 
\Veti{X_{2}^2 \,e_{1}} & 
\Veti{X_{2}X_{3}\,e_{1}} \ 
\Veti{X_{1}^2 X_2 \,e_{3}} & 
\Veti{X_{3}^{2}\,e_{1}} \ 
\Veti{X_{1}^2 X_3 \,e_{3}} & 
& 
\Veti{X_{1}X_{2}^2 \,e_{3}} & 
\Veti{X_{1}X_2X_3\,e_{3}} & 
\Veti{X_{1}X_3^2\,e_{3}} & 
\Veti{e_{2}} & 
\Veti{X_{2}^{3}\,e_{3}} & 
\Veti{X_{2}^2 X_{3} \,e_{3}} & 
\Veti{X_2X_{3}^{2}\,e_{3}} & 
\Veti{X_{3}^3\,e_{3}} 
\end{array}
$$
}

\noindent
Le nombre de formes déterminantales pures non nulles est le produit des
cardinaux des ensembles de divisibilité, ici $1^{11} \times 2^3 = 8$.
Par exemple, celle bâtie en sélectionnant, 
dans les colonnes trois, cinq, six du tableau ci-dessus, 
les indices de divisibilité $3$, $3$, $1$ 
est $p_1^4\,p_2\,p_3^9\,(X_{1}X_{2}^{3})^\star$ 
(dans le dessin ci-après, 
la colonne $\sbullet$ est la \og la colonne argument \fg{}).
{\small
$$
\EastBordermatrix{
p_{1} & . & . & . & . & . &\sbullet & . & . & . & . & . & . & . & . & \Heti{X_{1}^{4}} \\ 
. & p_{1} & . & . & . & . &\sbullet & . & . & . & . & . & . & . & . & \Heti{X_{1}^{3}X_{2}} \\ 
. & . & p_{3} & . & . & . &\sbullet & . & . & . & . & . & . & . & . & \Heti{X_{1}^{3}X_{3}} \\ 
. & . & . & p_{1} & . & . &\sbullet & . & . & . & . & . & . & . & . & \Heti{X_{1}^{2}X_{2}^{2}} \\ 
. & . & . & . & p_{3} & . &\sbullet & . & . & . & . & . & . & . & . & \Heti{X_{1}^{2}X_{2}X_{3}} \\ 
. & . & . & . & . & p_{1} &\sbullet & . & . & . & . & . & . & . & . & \Heti{X_{1}^{2}X_{3}^{2}} \\
. & . & . & . & . & . &\sbullet & . & . & . & . & . & . & . & . & \Heti{X_{1}X_{2}^{3}} \mouton \\
. & . & . & . & . & . &\sbullet &  p_{3} & . & . & . & . & . & . & . & \Heti{X_{1}X_{2}^{2}X_{3}} \\ 
. & . & . & . & . & . &\sbullet & . & p_{3} & . & . & . & . & . & . & \Heti{X_{1}X_{2}X_{3}^{2}} \\ 
. & . & . & . & . & . &\sbullet & . & . & p_{3} & . & . & . & . & . & \Heti{X_{1}X_{3}^{3}} \\ 
. & . & . & . & . & . &\sbullet & . & . & . & p_{2} & . & . & . & . & \Heti{X_{2}^{4}} \\ 
. & . & . & . & . & . &\sbullet & . & . & . & . & p_{3} & . & . & . & \Heti{X_{2}^{3}X_{3}} \\ 
. & . & . & . & . & . &\sbullet & . & . & . & . & . & p_{3} & . & . & \Heti{X_{2}^{2}X_{3}^{2}} \\ 
. & . & . & . & . & . &\sbullet & . & . & . & . & . & . & p_{3} & . & \Heti{X_{2}X_{3}^{3}} \\ 
. & . & . & . & . & . &\sbullet & . & . & . & . & . & . & . & p_{3} & \Heti{X_{3}^{4}} \\ 
}
$$
}

Lorsque l'on prend comme mécanisme de sélection celui associé à
$\minDiv$ \idest{} celui qui consiste à choisir le plus petit indice de
divisibilité, la forme linéaire obtenue $\bfA[\uX]_\delta \to \bfA$
prend le nom de $\omega_{\uP}$ (cf.~\ref{DefEndoOmega}). Ainsi,
dans le tableau ci-dessus des indices de divisibilité de $D=(2,4,1)$,
en sélectionnant dans les colonnes où il y a un choix (colonnes 3,5,6)
les indices de divisibilité minimum $1,1,1$, on obtient $\omega_{\pXD}
= p_1^6 \, p_3^7\, p_2 \, (X_{1}X_{2}^{3})^\star$.
Plus généralement, on a 
$$
\omega_{\pXD} = \pi\times (X^\emouton)^\star  \qquad \hbox {avec} \qquad
\pi = \prod_{X^\alpha \in \Jex_{1,\delta}} p_{\minDiv(X^\alpha)}
$$

\label{omegaCasParticulier} 
\label{NOTA04-omega}%
\index{forme linéaire!$\omega=\omega_\uP$ (pilotée par $\minDiv$)}%

\begin{proof}\leavevmode
Cette démonstration présente un certain nombre d'analogies avec celle
en degré $\geqslant \delta+1$ (proposition \ref{MacRaeJeuEtalon}) et
de ce fait nous détaillerons moins.

i) Confer~\ref{LemmeDenombrementJ1moins2id}.

\smallskip

ii) Dans le cas du jeu étalon généralisé, il est inutile d'invoquer \ref{rankSyldeltaDirectProof}:
tout mineur plein de $\Syl_\delta$ possède une ligne nulle, celle correspondant
au \MoutonNoir{}, et en conséquence est nul.

\medskip
Vu la présence de la ligne nulle \MoutonNoir{}, toute forme déterminantale pure $\mu$ de
$\Syl_\delta$ est multiple de la forme linéaire $(X^{\emouton})^\star$,
le multiplicateur $\mu(X^\emouton)$ étant le mineur d'ordre $s_\delta$ défini par les
$s_\delta$ colonnes sélectionnées et les $s_\delta$ lignes indexées
par les monômes de $\Jex_{1,\delta}$.

Il s'agit donc d'étudier tous les mineurs d'ordre $s_\delta$ de ce
type. L'analyse réalisée en degré $d \geqslant \delta + 1$ s'applique et
conduit à deux types de mineurs.  Seuls ceux non nuls nous concernent
et sont associés à un mécanisme de sélection défini sur les monômes de
$\Jex_{1,\delta}$.

\smallskip  

iii) On utilise un argument analogue à celui du point correspondant de la
proposition \ref{MacRaeJeuEtalon}, à ceci près que le nombre de monômes
de $\Jex_{1\setminus2,\delta}^{(i)}$ est $\widehat d_i -1$.

\smallskip

iv) 
La forme linéaire  $\omegares = p_1^{\widehat d_1-1}\cdots p_n^{\widehat d_n-1}\,(X^\emouton)^\star$ est 
sans torsion.

D'après le point précédent, pour chaque $1 \le i\le n$, il y a une
forme linéaire pure du type $\fm_i \, \omegares$ avec $\fm_i$ monôme
en les $p_j$ ne dépendant pas de $p_i$.  La suite $(\fm_1, \dots,
\fm_n)$ est de profondeur $\geqslant 2$ (cela résulte du fait que, sur
un anneau quelconque, une famille finie de monômes premiers entre eux
est de profondeur $\geqslant 2$).  On en déduit que le pgcd fort de
ces $n$ formes déterminantales pures est $\omegares$.

Par ailleurs,  comme $\omegares$ divise \emph {chaque} forme déterminantale pure,
c'est le pgcd fort de $\DVect_{s_\delta}(\Syl_\delta)$.

\medskip
v)
D'une part, $\bfB_\delta$ est de rang 1. Et le fait que  
le sous-module $\DVect_{s_\delta}(\Syl_\delta)$ admette pour pgcd fort $\omegares$.
s'énonce en disant que $\bfB_\delta$ est un module de MacRae de rang $1$ 
d'invariant de MacRae $\uomegares$. 
\end{proof}

\subsubsection*{La forme linéaire $\omega^\sigma$ et le cofacteur $b_\delta^\sigma$
                dans le mécanisme de sélection $\sminDiv$}

Pour toute permutation $\sigma \in \mathfrak S_n$,
le mécanisme de sélection défini par le plus petit
indice de divisibilité au sens de la relation d'ordre $<_\sigma$
caractérisée par $\sigma(1) <_\sigma \cdots <_\sigma \sigma(n)$ conduit
à une forme linéaire déterminantale pure notée $\omega^\sigma_{\uP}$.
Ainsi
$$
\omega^\sigma_{\pXD} = \pi\times (X^\emouton)^\star  \qquad \hbox {avec} \qquad
\pi = \prod_{X^\alpha \in \Jex_{1,\delta}} p_{\sminDiv(X^\alpha)}
$$
Pour plus de détails, cf. la section~\ref{SigmaVersion}.

\label{NOTA04-omegasigma}%
\index{forme linéaire!$\omega^\sigma=\omega^\sigma_\uP$ (pilotée par $\sminDiv$)}%

\medskip

Pour alléger la notation $\omega_{\pXD}$, nous omettrons $\pXD$
ci-dessous (idem pour $\omega^\sigma$).

\begin{prop} [Où $\Jex_{2,\delta}$ apparaît dans le cofacteur de $\omegares$ dans $\omega$]
\label{deltaCofacteurJeuEtalonGen}
Pour le jeu étalon généralisé $(p_1X_1^{d_1}, \dots, p_nX_n^{d_n})$:

\begin{enumerate}[\rm i)]
\item 
Les deux formes linéaires
$\omega$ et $\omegares : \bfA[\uX]_\delta \to \bfA$ sont reliées par le
multiplicateur explicite $b_\delta$, qui est un monôme en $p_1,
\dots, p_n$ ne dépendant pas de $p_n$ :
$$
\omega =  b_\delta\ \omegares  
\qquad \text{avec }
b_\delta = \prod_{X^\alpha \in \Jex_{2,\delta}} p_{\minDiv(X^\alpha)} 
\ = \ 
\prod_{i=1}^n p_i^{\nu_i}
\qquad 
\text{ où $\nu_i = \dim \Jex_{2,\delta}^{(i)}$ ($\nu_n=0$).}
$$

\item
De manière analogue, pour une permutation $\sigma \in \mathfrak S_n$,
on dispose d'un monôme $b^\sigma_\delta$ en $p_1, \dots, p_n$ ne
dépendant pas de $p_{\sigma(n)}$:
$$
\omega^\sigma =  b_\delta^\sigma \ \omegares  
\qquad \text{avec }
b_\delta^\sigma = \prod_{X^\alpha \in \Jex_{2,\delta}} p_{\sminDiv(X^\alpha)} 
$$

\item 
Soit $\Sigma_2$ l'ensemble des $n$ transpositions $(i,n)$ pour 
$1 \leqslant i \leqslant n$.
Alors $\omegares$ est le pgcd fort des $(\omega^\sigma)_{\sigma \in \Sigma_2}$.
\end{enumerate}
\end {prop}

\label{NOTA04-bdelta}
\label{NOTA04-bdeltasigma}
%
%

\begin {proof} \leavevmode

\noindent
i)
Par définition de $\omegares$, la forme linéaire $\omega$ est multiple de $\omegares$.
Pour déterminer le multiplicateur, il suffit d'évaluer ces formes linéaires 
en le \MoutonNoir{} $X^\emouton$. 

On a $\omegares(X^\emouton) = p_1^{\widehat d_1 - 1} \cdots p_n^{\widehat d_n - 1}$ 
où $\widehat d_i - 1 = \dim\Jex_{1\setminus2,\delta}^{(i)}$.
Comme $\Jex_{1 \setminus 2,\delta} = \bigoplus_{i=1}^n \Jex_{1\setminus2,\delta}^{(i)}$, 
l'égalité précédente s'écrit 
$$
\omegares(X^\emouton) 
= \prod_{X^\alpha \in \Jex_{1\setminus 2,\delta}} p_{\minDiv(X^\alpha)} 
$$
Par construction, la forme linéaire $\omega$ vérifie :
$$
\omega(X^\emouton) 
\ = \ 
\prod_{X^\alpha \in \Jex_{1,\delta}} p_{\minDiv(X^\alpha)} 
$$
Comme $\Jex_{1,\delta} = \Jex_{1 \setminus 2,\delta} \oplus \Jex_{2,\delta}$
et $\Jex_{2,\delta} = \bigoplus\limits_{i=1}^n \Jex_{2,\delta}^{(i)}$, 
le multiplicateur $b_\delta$ vaut 
$$
b_\delta \ = \ 
\prod_{X^\alpha \in \Jex_{2,\delta}} p_{\minDiv(X^\alpha)} 
\ = \ 
\prod_{i=1}^n p_i^{\nu_i}
\qquad 
\text{ où $\nu_i = \dim \Jex_{2,\delta}^{(i)}$}
$$
Le fait que $b_\delta$ ne dépende pas de $p_n$ provient du fait
que tout monôme $X^\alpha$ de $\Jex_2$ ayant au moins deux indices de
divisibilité est tel que $\minDiv(X^\alpha) < n$, 
ce qui se traduit par $\Jex_{2,\delta}^{(n)} = 0$.

\medskip
\noindent
ii) Se traite comme i).

\medskip
\noindent
iii) Comme $b_\delta^\sigma$ est un monôme en $p_1, \dots, p_n$
ne dépendant pas de $p_{\sigma(n)}$, 
la famille $(b_\delta^\sigma)_{\sigma 
\in \Sigma_2}$ est de profondeur $\geqslant 2$.
\end {proof}

\subsection{Une approche structurelle en tout degré $d$}

Après la chasse aux mineurs dans les matrices de Sylvester, nous allons
procéder de manière totalement différente pour cerner la structure
du $\bfA$-module $\bfB_d$ du jeu étalon généralisé. Nous allons l'écrire
comme une somme directe de modules monogènes $\bfA/\fp_\alpha$ pour $|\alpha| = d$
où $\fp_\alpha = \langle p_i, i \in \DivSeq(X^\alpha)\rangle$, écriture sur 
laquelle les propriétés de profondeur se visualisent aisément.

Vont de nouveau intervenir les idéaux monomiaux de $\bfA[\uX]$ suivants :
$$
\Jex_0 = \bfA[\uX], \qquad\qquad
\Jex_1 = \langle X_1^{d_1}, \dots, X_n^{d_n} \rangle, \qquad \qquad 
\Jex_2 = \langle X_i^{d_i}X_j^{d_j} \mid i \neq j \rangle
$$
ainsi que les supplémentaires monomiaux $\Jex_{1\setminus 2}$ de
$\Jex_2$ dans $\Jex_1$, et $\Jex_{0 \setminus 1}$ de $\Jex_1$ dans
$\Jex_0$ (intervenu page~\pageref{NOTA02-Jex01}).
Pour tout~$d$, on a évidemment une somme directe 
$\bfA[\uX]_d = \Jex_{0\setminus 1,d} \oplus \Jex_{1,d}$
si bien que les deux $\bfA$-modules $\Jex_{0\setminus 1,d}$ et 
$\bfA[\uX]_d / \Jex_{1,d}$ sont isomorphes. 
Par conséquent, $\Jex_{0\setminus 1,\,d} = 0$ pour tout $d \geqslant \delta+1$
tandis que le $\bfA$-module $\Jex_{0\setminus 1,\,\delta}$ est monogène 
engendré par le \MoutonNoir{}  $X^{\emouton}$.
Avec ces notations, on a l'égalité évidente :
$$
\bfA[\uX]
\ = \ 
\Jex_{0\setminus 1} \,\oplus\, 
\Jex_{1\setminus 2} \,\oplus\, 
\Jex_{2}
$$
qui traduit le fait qu'un monôme est divisible soit par \textit{aucun} $X_i^{d_i}$, 
soit par \textit{un seul} $X_i^{d_i}$, soit par \textit{au moins deux} $X_i^{d_i}$.
Cette décomposition joue le rôle suivant :
pour $X^\alpha \in \bfA[\uX]$,
l'idéal $\fp_\alpha$ de $\bfA$, engendré par les $p_i$ tels que $\alpha_i \geqslant d_i$,
est 
$$
\left\{
\begin{array}{ll}
\text{nul } & \text{si $X^\alpha \in \Jex_{0\setminus 1}$} \\ [0.2cm]
\text{engendré par un élément régulier} & \text{si $X^\alpha \in \Jex_{1\setminus 2}$} \\  [0.2cm]
\text{de profondeur $\geqslant 2$} & \text{si $X^\alpha \in \Jex_2$} 
\end{array}
\right.
$$
Et pour être totalement complet et avoir un avant-goût de la preuve,
nous allons considérer la décomposition déjà intervenue: $\Jex_{1\setminus 2} =
\bigoplus\limits_{i=1}^n \Jex_{1\setminus 2}^{(i)}$, où
$\Jex_{1\setminus 2}^{(i)}$ est le $\bfA$-module engendré par les
$X^\alpha \in \Jex_{1\setminus 2}$ tels que $\DivSeq(X^\alpha) =
\{i\}$.  On apporte ainsi une précision à la monogénéité de
$\fp_\alpha$: pour $X^\alpha \in \Jex_{1\setminus 2}^{(i)}$, on a
$\fp_\alpha = \langle p_i \rangle$.

\index{supplémentaire!monomial excédentaire}
%
%

\begin{prop} \label{BdJeuEtalonGen}

Soit $\pXD = \langle p_1 X_1^{d_1},\dots, p_n X_n^{d_n}\rangle$ le jeu étalon généralisé 
et $\bfB = \bfA[\uX]/\langle \pXD\rangle$.

\begin{enumerate}[\rm i)]
\item 
On a l'égalité 
$$
\langle \pXD \rangle \ = \ 
\bigoplus_{X^\alpha \in \bfA[\uX]} 
\fp_\alpha X^\alpha
\qquad 
\text{où}\quad \fp_\alpha = \langle p_i,\, i \in \DivSeq(X^\alpha)\rangle
$$
\item 
On  a l'isomorphisme canonique de $\bfA$-modules : 
$$
\forall\, d, \qquad 
\bfB_d
\ \simeq \
\!\!\!
\bigoplus_{X^\alpha \in \bfA[\uX]_d}
\!\!\bfA/\fp_{\alpha} 
$$
\item Pour $d \geqslant \delta$, l'isomorphisme précédent se simplifie :
$$
\forall\, d \geqslant \delta+1, \quad 
\bfB_d \ \simeq \ 
\bigoplus_{i=1}^n \big(\bfA/\langle p_i \rangle\big)^{\widehat d_i}
\ \oplus \ 
\bigoplus_{X^\alpha \in \Jex_{2,d}} \bfA /\fp_\alpha
$$
et pour $d = \delta$ : 
$$
\bfB_\delta 
\ \simeq \ 
\bfA X^\emouton 
\ \oplus \ 
\bigoplus_{i=1}^n \big(\bfA/\langle p_i \rangle\big)^{\widehat d_i - 1}
\ \oplus \ 
\bigoplus_{X^\alpha \in \Jex_{2,\delta}} \bfA /\fp_\alpha
$$

\item 
Pour $d\geqslant \delta+1$, le $\bfA$-module $\bfB_d$ est de MacRae de rang $0$ et 
$\MacRae(\bfB_d)$ est le sous-module de $\bfA$ engendré par le scalaire
$p_1^{\widehat d_1}\cdots p_n^{\widehat d_n}$.

Pour $d = \delta$, le $\bfA$-module $\bfB_\delta$ est de MacRae de rang $1$ et 
$\MacRaeVect(\bfB_\delta)$ est le sous-module de $\bfB_\delta^\star$ 
engendré par la forme linéaire 
$\omegares = p_1^{\widehat d_1 -1}\cdots p_n^{\widehat d_n -1} (X^\emouton)^\star$
passée au quotient.
\end{enumerate}
\end{prop}

\begin{proof}
i) 
L'inclusion $\supset$ est facile.
Soit $\alpha$ fixé et prenons un élément de $\fp_\alpha X^\alpha$.
Il s'écrit comme combinaison $\bfA$-linéaire de $p_i X^\alpha$ pour $i \in \DivSeq(X^\alpha)$;
par définition, on a $X_i^{d_i} \mid X^\alpha$, en conséquence
$p_iX^\alpha = (X^\alpha/X_i^{d_i})\, p_i X_i^{d_i} \in \langle\pXD\rangle$. 

Reste à prouver l'autre inclusion. 
C'est le caractère monomial dans l'anneau
$\bfk[p_1, \dots, p_n, X_1, \dots, X_n]$ qui l'assure de la manière suivante. 
Un polynôme $F$ de cet anneau s'écrit 
$\sum \lambda_{\alpha, \beta} 
p^\beta X^\alpha$.
Supposons que~$F$ appartienne à l'idéal $\langle\pXD\rangle$ ;
la propriété monomiale en question affirme que 
$\lambda_{\alpha,\beta} =0$ à l'exception des $(\alpha, \beta)$
vérifiant $\alpha_i \geqslant d_i$ et $\beta_i \geqslant 1$ pour un certain $i$.
Dans ce dernier cas, en rappelant que $\bfA = \bfk[p_1, \dots, p_n]$, on a 
$p^\beta X^\alpha \in \bfA p_i X^\alpha$ avec $X_i^{d_i} \mid X^\alpha$, 
c'est-à-dire $p^\beta X^\alpha \in \fp_\alpha X^\alpha$.
Ainsi, $F \in \bigoplus\limits_{X^\alpha \in \bfA[\uX]} \fp_\alpha X^\alpha$.

ii) Il suffit d'écrire de manière tautologique :
$\bfA[\uX]_d = \bigoplus\limits_{|\alpha| = d}  \bfA X^\alpha$ 
et de quotienter par la composante homogène de degré $d$ de $\langle \pXD \rangle$ 
qui vaut $\bigoplus\limits_{|\alpha| = d} \fp_\alpha X^\alpha$.

iii)
Avec ii) et la décomposition de $\bfA[\uX]$, on a :
$$
\bfB_d
\ \simeq \ 
\bigoplus_{X^\alpha \in \Jex_{0\setminus 1,d}}
\!\!\bfA/\fp_\alpha 
\ \oplus \ 
\bigoplus_{X^\alpha \in \Jex_{1\setminus 2,d}}
\!\!\bfA/\fp_\alpha 
\ \oplus \ 
\bigoplus_{X^\alpha \in \Jex_{2,d}}
\!\!\bfA/\fp_\alpha 
$$
Pour $X^\alpha\in \Jex_{0\setminus 1}$, on a $\fp_\alpha = 0$.
Pour le membre du milieu, écrivons $\Jex_{1\setminus 2} = \bigoplus_{i=1}^n \Jex_{1\setminus 2}^{(i)}$ ;
si $X^\alpha \in \Jex_{1\setminus 2}^{(i)}$, on~a $\fp_\alpha  = \langle p_i \rangle$.
Enfin, pour $X^\alpha \in \Jex_2$, on a $\fp_\alpha \supset \langle p_i, p_j \rangle$ avec $i \neq j$ ;
en particulier $\Gr(\fp_\alpha) \geqslant 2$. Ainsi:
$$
\bfB_d
\quad \simeq \quad 
\bfA^{\dim \Jex_{0\setminus 1, d}}
\ \oplus \ 
\bigoplus_{i=1}^n
\big(\bfA/\langle p_i \rangle\big)^{\dim \Jid}
\ \oplus \ 
\bigoplus_{X^\alpha \in \Jex_{2,d}}
\!\!\bfA/\fp_\alpha 
$$
De plus, on a 
$$
\dim \Jex_{0\setminus 1, d}
\ = \ 
\left\{
\begin{array}{ll}
0 & \text{pour $d \geqslant \delta+1$} \\
[0.2cm]
1 & \text{pour $d = \delta$} 
\end{array}
\right.
\qquad \text{ et } \qquad 
\dim \Jex_{1\setminus 2, d }^{(i)}
\ = \ 
\left\{
\begin{array}{ll}
\widehat d_i & \text{si $d \geqslant \delta+1$} \\ [0.2cm]
\widehat d_i -1 & \text{si $d = \delta$} 
\end{array}
\right.
$$

iv) D'après iii), on a 
$$
\bfB_d  \ =\ 
\bigoplus_{i=1}^n \big(\bfA/\langle p_i \rangle\big)^{\widehat d_i}
\, \oplus \, 
\bigoplus_{X^\alpha \in \Jex_{2,d}} \bfA /\fp_\alpha
\qquad \text{module du type } \qquad 
\displaystyle \bigoplus_{a \in \mathscr A} \bfA/\langle a \rangle 
\oplus \bigoplus_{\ub \in \mathscr B} \bfA/\langle \ub \rangle
$$
avec $\mathscr A$ constitué de $p_1, \dots, p_1, p_2, \dots, p_2, \dots, p_n$
(où chaque $p_i$ intervient $\widehat d_i$ fois), 
et $\mathscr B$ constitué des $\ub_\alpha$ système générateur de $\fp_\alpha$.
D'après~\ref{ModuleTresDecompose}, on a $\MacRae(\bfB_d) = \prod\limits_{a\in \mathscr A} a = 
p_1^{\widehat d_1}\cdots p_n^{\widehat d_n}$.

D'après le point iii), 
$\bfB_\delta$ s'écrit $\bfA X^\emouton \oplus M_0$ 
avec 
$M_0 = 
\displaystyle \bigoplus_{i=1}^n \big(\bfA/\langle p_i \rangle\big)^{\widehat d_i - 1}
\ \oplus \ 
\bigoplus_{X^\alpha \in \Jex_{2,\delta}} \bfA /\fp_\alpha$.
Pour les mêmes raisons que précédemment, 
$M_0$ est de MacRae de rang $0$ et $\MacRae(M_0) = 
p_1^{\widehat d_1 -1}\cdots p_n^{\widehat d_n -1}$.
D'après~\ref{AoplusM0}, on en déduit que $\bfB_\delta$ est de MacRae de 
rang $1$ et $\MacRaeVect(\bfB_\delta) 
= p_1^{\widehat d_1 -1}\cdots p_n^{\widehat d_n -1} (X^\emouton)^\star$.
\end{proof}

On pourrait résumer ce qui précède en disant que 
\begin{itemize}
\item
$\Jex_{0\setminus 1,d}$ porte le rang de $\bfB_d$, 
\item 
$\Jex_{1\setminus 2,d}$ porte l'invariant de MacRae vectoriel de $\bfB_d$ 
\item 
$\Jex_{2,d}$ porte les cofacteurs, c'est-à-dire les éléments de l'idéal 
$\fb = \calF_0(\bfB_d)/\MacRae(\bfB_d)$.
\end{itemize}

\bigskip

Illustrons la décomposition du module $\bfB_d$ 
pour $D = (2,1,3)$ et $d = \delta +1 = 4$.
Pour cela, commençons par lister les 
$11 = \widehat d_1 + \widehat d_2 + \widehat d_3 = 3 + 6 + 2$ monômes de 
$\Jex_{1\setminus 2, d}$, 
puis les $4$ monômes de $\Jex_{2,d}$.
Une fois cela fait, considérons les $11$ colonnes $(X^\alpha/X_i^{d_i})\, e_i$ 
avec $X^\alpha \in \Jex_{1\setminus 2, d}$, 
puis les autres colonnes restantes (ici $8$).

Ce rangement fournit le nouveau visage de $\Syl_d(\pXD)$ du jeu étalon généralisé :
$$
\renewcommand \Heti[1] {\omit\quad \mbox{\scriptsize$#1$} \hfil} 
\NorthEastBordermatrix{
\Veti{X_{1}^{2}\, e_1} &  
\Veti{X_{1}X_{3}\,e_{1}} &
\Veti{X_{3}^{2}\,e_{1}} &
\Veti{X_{1}X_{2}^{2}\,e_{2}} & 
\Veti{X_{1}X_{2}X_{3}\,e_{2}} & 
\Veti{X_{1}X_{3}^{2}\,e_{2}} & 
\Veti{X_{2}^{3}\,e_{2}} & 
\Veti{X_{2}^{2}X_{3}\,e_{2}} & 
\Veti{X_{2}X_{3}^{2}\,e_{2}} &
\Veti{X_{1}\,e_{3}} & 
\Veti{X_{3}\,e_{3}} & 
\VR 
\Veti{X_{1}X_{2}\,e_{1}} & 
\Veti{X_{1}^{3}\,e_{2}} & 
\Veti{X_{2}^{2}\,e_{1}} & 
\Veti{X_{1}^{2}X_{2}\,e_{2}} & 
\Veti{X_{2}X_{3}\,e_{1}} & 
\Veti{X_{1}^{2}X_{3}\,e_{2}} & 
\Veti{X_{3}^{3}\,e_{2}} & 
\Veti{X_{2}\,e_{3}} & 
\\
p_{1}& .& .& .& .& .& .& .& .& .& .& \VR .& .& .& .& .& .& .& . & \Heti{X_{1}^{4}} \\ 
.& p_{1}& .& .& .& .& .& .& .& .& .& \VR .& .& .& .& .& .& .& . & \Heti{X_{1}^{3}X_{3}} \\ 
.& .& p_{1}& .& .& .& .& .& .& .& .& \VR .& .& .& .& .& .& .& . & \Heti{X_{1}^{2}X_{3}^{2}} \\ 
.& .& .& p_{2}& .& .& .& .& .& .& .& \VR .& .& .& .& .& .& .& . & \Heti{X_{1}X_{2}^{3}} \\ 
.& .& .& .& p_{2}& .& .& .& .& .& .& \VR .& .& .& .& .& .& .& . & \Heti{X_{1}X_{2}^{2}X_{3}} \\ 
.& .& .& .& .& p_{2}& .& .& .& .& .& \VR .& .& .& .& .& .& .& . & \Heti{X_{1}X_{2}X_{3}^{2}} \\ 
.& .& .& .& .& .& p_{2}& .& .& .& .& \VR .& .& .& .& .& .& .& . & \Heti{X_{2}^{4}} \\ 
.& .& .& .& .& .& .& p_{2}& .& .& .& \VR .& .& .& .& .& .& .& . & \Heti{X_{2}^{3}X_{3}} \\ 
.& .& .& .& .& .& .& .& p_{2}& .& .& \VR .& .& .& .& .& .& .& . & \Heti{X_{2}^{2}X_{3}^{2}} \\ 
.& .& .& .& .& .& .& .& .& p_{3}& .& \VR .& .& .& .& .& .& .& . & \Heti{X_{1}X_{3}^{3}} \\ 
.& .& .& .& .& .& .& .& .& .& p_{3}& \VR .& .& .& .& .& .& .& . & \Heti{X_{3}^{4}} \\ 
\HR{20}
.& .& .& .& .& .& .& .& .& .& .& \VR p_{1}& p_{2}& .& .& .& .& .& . & \Heti{X_{1}^{3}X_{2}} \\ 
.& .& .& .& .& .& .& .& .& .& .& \VR .& .& p_{1}& p_{2}& .& .& .& . & \Heti{X_{1}^{2}X_{2}^{2}}\\ 
.& .& .& .& .& .& .& .& .& .& .& \VR .& .& .& .& p_{1}& p_{2}& .& . & \Heti{X_{1}^{2}X_{2}X_{3}} \\ 
.& .& .& .& .& .& .& .& .& .& .& \VR .& .& .& .& .& .& p_{2}& p_{3} & \Heti{X_{2}X_{3}^{3}} \\ 
}
$$
On a l'égalité $\MacRae(\bfB_d) = p_1^3 p_2^6 p_3^2$. L'idéal cofacteur 
$\fb = \calF_0(\bfB_d)/\MacRae(\bfB_d)$ est
engendré par tous les mineurs pleins de la matrice Sud-Est ci-dessus 
(\idest{} par ses mineurs de taille 4); il contient l'idéal 
$\langle p_1^3 p_2, p_1^3p_3, p_2^4 \rangle$ de profondeur $2$
(le radical de ce dernier idéal est $\langle p_1p_3,p_2\rangle$).
L'idéal cofacteur $\fb$ est exactement l'idéal
produit $\langle p_1,p_2\rangle^3 \langle p_1,p_3\rangle$ des
\og idéaux par paliers\fg{} (voir remarque \ref{MatricePaliers}).

\subsection{Le module $\bfB'_\delta=\bfA[\protect\uX]_\delta/\langle\protect\uP,\nabla\rangle_\delta$
            du jeu étalon généralisé}

Rappelons que le $\bfA$-module $\bfB'_\delta$ associé à un système $\uP$
est le quotient $\bfB_\delta / \bfA \overline \nabla$ où $\nabla$ est
le (un) déterminant bezoutien de $\uP$. 
On verra (cf.~proposition~\ref{Rank1to0MacRae} et théorème~\ref{PoidsNormalisationMacRae}) 
que les deux modules de MacRae~$\bfB'_\delta$ et~$\bfB_\delta$, 
le premier de rang 0, le second de rang 1, ont leurs
invariants de MacRae imbriqués et que le lien entre les deux
est assuré par le bezoutien $\nabla$. Dans le cas du jeu étalon généralisé,
on fournit un argument direct : c'est l'objet du résultat qui vient 
et la remarque après la preuve.

\begin{prop} \label{BprimedeltaJeuEtalon}
Le $\bfA$-module $\bfB'_\delta$ du jeu étalon généralisé 
est de MacRae de rang $0$
et son invariant de MacRae $\MacRae(\bfB'_\delta)$ 
est (l'idéal principal de
$\bfA$ engendré par) le scalaire régulier 
$p_1^{\widehat d_1}\cdots p_n^{\widehat d_n}$.
\end{prop}

\begin{proof}
Le polynôme $\nabla = p_1 \cdots p_n X^{\emouton}$ est un bezoutien du jeu étalon généralisé.
Or, d'après~\ref{BdJeuEtalonGen}
$$
\bfB_\delta 
\ \simeq \ 
\bfA X^\emouton 
\ \oplus \ 
\bigoplus_{i=1}^n \big(\bfA/\langle p_i \rangle\big)^{\widehat d_i - 1}
\ \oplus \ 
\bigoplus_{X^\alpha \in \Jex_{2,\delta}} \bfA /\fp_\alpha
$$
Pour obtenir $\bfB'_\delta = \bfB_\delta / \overline{\nabla}$, 
il s'agit donc de quotienter le membre droit 
par $p_1 \cdots p_n X^{\emouton} \oplus 0 \oplus 0$.
Ainsi 
$$
\bfB'_\delta 
\ \simeq \ 
\bfA/\langle p_1 \cdots p_n \rangle
\ \oplus \ 
\bigoplus_{i=1}^n \big(\bfA/\langle p_i \rangle\big)^{\widehat d_i - 1}
\ \oplus \ 
\bigoplus_{X^\alpha \in \Jex_{2,\delta}} \bfA /\fp_\alpha
\quad \text{ du type } \quad 
\displaystyle \bigoplus_{a \in \mathscr A} \bfA/\langle a \rangle 
\oplus \bigoplus_{\ub \in \mathscr B} \bfA/\langle \ub \rangle
$$
avec $\mathscr A$ constitué des éléments réguliers 
$$
p_1\cdots p_n, \qquad 
p_1,\dots, p_1 \text{ ($\widehat d_1 - 1$ fois)}, 
\qquad \dots \qquad 
p_n, \dots, p_n \text{ ($\widehat d_n - 1$ fois)}
$$
et $\mathscr B$ constitué des systèmes générateurs de $\fp_\alpha$, suites de 
profondeur $\geqslant 2$.
D'après~\ref{ModuleTresDecompose}, 
ce module est de MacRae de rang $0$ et 
son invariant de MacRae est engendré par 
$p_1 \cdots p_n  \times p_1^{\widehat d_1 -1}\cdots p_n^{\widehat d_n -1}$
qui vaut 
$p_1^{\widehat d_1}\cdots p_n^{\widehat d_n}$. 
\end{proof}

\begin{rmq} \label{Rank1to0MacRaeJeuEtalonGen}
L'invariant de MacRae $\MacRae(\bfB'_\delta)$ est obtenu en évaluant la
forme linéaire $\uomegares = \MacRaeVect(\bfB_\delta)$ en $\overline
\nabla$.  En effet, cette dernière est $p_1^{\widehat d_1 -1}\cdots
p_n^{\widehat d_n -1} (X^\emouton)^\star$ et son évaluation en $\nabla = p_1
\cdots p_n X^{\emouton}$ est bien $\MacRae(\bfB'_\delta) =
p_1^{\widehat d_1}\cdots p_n^{\widehat d_n}$.
\end{rmq}

\subsection{Annulateurs des modules $(\bfB_d)_{d \geqslant \delta+1}$ et $\bfB'_\delta$ du jeu étalon généralisé}

La première égalité (des idéaux annulateurs) dans le résultat qui suit est
un cas particulier de~\ref{AnnEqualities} mais nous reprenons 
la preuve en charge pour le jeu étalon généralisé. Figure d'autre part le
produit $p_1 \cdots p_n$ comme générateur de ces idéaux, scalaire
rarement égal au résultant $p_1^{\widehat d_1}\cdots p_n^{\widehat d_n}$.
Nous reviendrons sur une problématique voisine : 
à quelle condition le résultant engendre-t-il
l'idéal d'élimination ?  
\`A cette occasion, on rappelle que
l'idéal d'élimination est égal à $\Ann(\bfB'_\delta)$
lorsque $\uP$ est régulière, cf. la proposition~\ref{AnnEqualities}.

\begin {lem}
Pour le jeu étalon généralisé, les modules $\bfB'_\delta$ et $\bfB_d$ pour
$d \geqslant \delta+1$,
de rang 0, ont même annulateur:
$$
\Ann(\bfB'_\delta) = \Ann(\bfB_d) = \langle p_1 \cdots p_n \rangle
$$
\end {lem}

\begin {proof} \leavevmode

Vérifions d'abord que le produit $\pi = p_1 \cdots p_n$ est dans
chaque annulateur. Pour $\bfB_d$, cela signifie que $\pi X^\alpha \in
\langle p_1X_1^{d_1},\dots, p_nX_n^{d_n}\rangle_d$ pour $|\alpha| = d$,
ce qui est vérifié puisque $X^\alpha \in \langle
X_1^{d_1},\dots, X_n^{d_n}\rangle_d$, conséquence de $d \geqslant \delta + 1$.

Pour $\bfB'_\delta$, en prenant le bezoutien $\nabla = \pi
X^\emouton$, on doit montrer que $\pi X^\alpha \in \langle \nabla,
p_1X_1^{d_1},\dots, p_nX_n^{d_n}\rangle_\delta$ pour $|\alpha| =
\delta$.  C'est encore vérifié puisque $X^\alpha \in \langle
X^\emouton, X_1^{d_1},\dots, X_n^{d_n}\rangle_\delta$.

\medskip
  
Réciproquement, soit $a \in \Ann(\bfB_d)$. En particulier, pour chaque
$i$, on a  
$aX_i^d \in \langle p_1X_1^{d_1}, \dots, p_nX_n^{d_n}\rangle_d$
donc $aX_i^d$ est multiple de $p_iX_i^{d_i}$ puis $p_i \mid a$. Ceci
ayant lieu pour tout $i$, on a $\pi \mid a$. De même, soit
$a \in \Ann(\bfB'_\delta)$; en particulier $aX_i^\delta \in \langle
\pi X^\emouton, p_1X_1^{d_1}, \cdots, p_nX_n^{d_n}\rangle_\delta$ donc
$aX_i^\delta$ est multiple de $p_iX_i^{d_i}$. On conclut comme précédemment.
\end {proof}

\cleardoublepage

\section{Notions et outils associés à l'application de Sylvester $\Syl_d(\protect \uP)$}
\label{ChapObjetsSylvester}

\subsection {Vue d'ensemble sur divers objectifs de ce chapitre}

Nous allons à présent regrouper et consolider certaines constructions
des chapitres précédents utilisées de manière éparse et dans des cas particuliers. Ici le
contexte est général: celui d'un système $\uP = (P_1, \dots, P_n)$ de
polynômes homogènes de $\bfA[\uX]$ de format $D = (d_1, \dots,
d_n)$. C'est la première différentielle
$\partial_1(\uP)$ du complexe de Koszul de $\uP$ \idest{} 
l'application $\bfA[\uX]$-linéaire qui porte le nom d'application de
Sylvester de $\uP$:
$$
\Syl = \Syl(\uP) : \, 
\xymatrix @C=1.5cm @M=0.4pc{
\bfA[\uX]^n \ar[r]^-{
\left [
\begin{smallmatrix}
P_1 & \cdots & P_n
\end{smallmatrix}
\right]
} & \bfA[\uX] 
}
$$
qui intervient par ses composantes homogènes $\Syl_d = \Syl_d(\uP)$ de degré $d$.
Celle-ci est l'application $\bfA$-linéaire :
$$
\Syl_d : 
\begin{array}[t]{rcl}
\rmK_{1,d}=
\displaystyle \bigoplus_{i=1}^n  \bfA[\uX]_{d - d_i}\,e_i & \longrightarrow & \rmK_{0,d}=\bfA[\uX]_d \vspace{0.1cm} \\
U_1 e_1 + \cdots + U_n e_n & \longmapsto & \displaystyle U_1P_1 + \cdots +  U_nP_n
\end{array}
$$
Si on se permet de noter encore $\Syl_d$ la matrice dans les bases
monomiales, on dispose donc d'une égalité de matrices lignes 
à coefficients dans $\bfA[\uX]$ :
$$
\Big[\ 
(X^\alpha)_{| \alpha | = d-d_1}\ P_1
\quad \big | \quad
\cdots 
\quad \big | \quad 
(X^\alpha)_{| \alpha | = d-d_n}\ P_n
\ \Big]
\ \ = \ \ 
\Big[\ 
(X^\alpha)_{| \alpha | = d} 
\ \Big]
\ \Syl_d
$$
Ce sont les idéaux monomiaux $\Jex_1, \Jex_2$ qui vont jouer un rôle important dans
l'indexation de certains mineurs de~$\Syl_d$.

\medskip
\noindent
$\bullet$
Dans un premier temps, nous allons expliciter 
divers scalaires appartenant à l'idéal d'élimination
$$
\ElimIdeal
\ = \ 
\Big \{
a \in \bfA \ \big | \ 
\exists\, e \in \bbN,\ 
\forall\, |\alpha| = e,\, a X^\alpha \in \langle \uP \rangle
\Big\}
$$
$\blacktriangleright$
Une première observation : tout mineur plein
(en lignes) de $\Syl_d(\uP)$ appartient à cet idéal
d'élimination. Justifions-le en invoquant une propriété générale: si
$u : E \to F$ est une application linéaire entre deux modules libres,
tout mineur plein $a$ de $u$ possède la propriété $a F \subset \Im
u$. Pour le voir, considérons des bases. Notons $M$ la matrice
restriction de $u$ au sous-espace de $E$ engendré par les vecteurs de
base correspondants au mineur; alors, pour tout $y \in F$, on a
$\det(M) y = M(\widetilde My)$ de sorte que 
$\det(M) y \in \Im M \subset \Im u$.

\medskip
Appliqué à $\Syl_d$, on obtient, pour tout mineur plein $a$ de $\Syl_d$, 
l'appartenance $a X^\alpha \in \langle\uP\rangle_d$ pour tout $|\alpha|=d$, impliquant que $a \in \ElimIdeal$.  Pour que cela soit vraiment intéressant,
il faut que, d'une part, il y ait des mineurs pleins \idest{} que la
matrice $\Syl_d$ soit plus \og large que haute\fg{}. Et que d'autre
part, ces mineurs pleins ne soient pas tous nuls et pourquoi pas, sous
certaines conditions portant sur $\uP$, soient des éléments réguliers de $\bfA$.

C'est là qu'intervient encore et toujours la barrière $d \geqslant \delta+1$.
Dans ce cas là, on a l'inégalité 
$\dim\rmK_{1,d} \geqslant \dim\rmK_{0,d}$ assurant l'existence de mineurs pleins.
Une telle inégalité dimensionnelle s'obtient via des mécanismes
de sélection $\rmK_{0,d} \hookrightarrow \rmK_{1,d}$,
cf. \ref{soussectionJeuEtalonGenDeltaPlus1}.
Mais nous avons besoin de bien plus que cette inégalité (la suite
permettra de mieux comprendre cette phrase).

\medskip

\noindent
$\blacktriangleright$
Et pour $d \leqslant \delta$, qu'en est-il? Pour cette plage, c'est
uniquement le degré $d = \delta$ qui intervient (à l'intérieur de
ce chapitre). On a vu en \ref{rangSyldelta} que tout mineur plein de
$\Syl_\delta$ est nul.  Ce sont les mineurs d'ordre $s_\delta = \dim
\bfA[\uX]_\delta -1$ qui jouent un rôle important via leur organisation
en formes linéaires déterminantales pures $\bfA[\uX]_\delta \to \bfA$
de $\Syl_\delta$.  En évaluant ces formes linéaires en un bezoutien
$\nabla$ de $\uP$, et plus généralement en un habitant de $\langle \uP
\rangle_\delta^\sat$, on obtient des scalaires dans l'idéal
d'élimination.

\medskip

\noindent $\bullet$
Un deuxième objectif réside dans la normalisation de \emph {certains} mineurs de Sylvester
et formes linéaires déterminantales. Il s'agit d'une part de contrôler
leur signe, leur spécialisation (en le jeu étalon, en le jeu étalon généralisé)
et d'autre part leur propriété de régularité.

\medskip

En ce qui concerne la normalisation, l'objectif ici est de définir \textit{précisément}
certains mineurs (pleins ou pas)  de $\Syl_d$ (pour tout $d$ et tout~$\uP$) ; 
\og précisément \fg{} voulant dire qu'il n'y a pas 
d'ambiguïté sur le signe. A cet effet, on introduira des 
\textit{endomorphismes} attachés aux idéaux monomiaux
$\Jex_1, \Jex_2$ permettant de considérer leur déterminants
$\det W_{1,d}(\uP)$ et $\det W_{2,d}(\uP)$ et les premières
propriétés de divisibilité.

\smallskip
On définira également une forme linéaire déterminantale au
sens de~\ref{soussection-d-egal-delta} (nous y reviendrons
en~\ref{DefDVect} dans un contexte hors élimination).
C'est la forme linéaire~$\omega : \bfA[\uX]_\delta \to \bfA$
mentionnée à plusieurs reprises, accompagnée de ses cousines
$\omega^\sigma$ pour $\sigma \in \fS_n$.

\medskip

En ce qui concerne la régularité, il est important de comprendre que
les mineurs pleins de Sylvester n'ont pas tous le même rôle.  Prenons
par exemple le format $D = (1,1,2)$ et le système
$$
\left\{
\begin{tabular}{rcp{10cm}} 
$P_{1}$ & $=$ & $pX + qY + rZ$\\ [0.1cm] 
$P_{2}$ & $=$ & $sX + tY + uZ$\\ [0.1cm] 
$P_{3}$ & $=$ & $aX^{2} + bXY + cXZ + dY^{2} + eYZ + fZ^{2}$\\ [0.1cm] 
\end{tabular}
\right.  
$$
Alors la matrice de Sylvester en degré $\delta+1 = 2$ est 
de format $6 \times 7$ :
$$
\Syl_{\delta+1} =
\NorthEastBordermatrix{
\Veti{X\,e_{1}} &\Veti{Y\,e_{1}} &\Veti{Z\,e_{1}} &\Veti{X\,e_{2}} &\Veti{Y\,e_{2}} &\Veti{Z\,e_{2}} &\Veti{e_{3}} & \\
p & . & . & s & . & . & a & \Heti{X^{2}} \\
q & p & . & t & s & . & b & \Heti{XY} \\
r & . & p & u & . & s & c & \Heti{XZ} \\
. & q & . & . & t & . & d & \Heti{Y^{2}} \\
. & r & q & . & u & t & e & \Heti{YZ} \\
. & . & r & . & . & u & f & \Heti{Z^{2}} \\
}
$$
Le mineur obtenu en supprimant la dernière colonne est nul. Ceci peut
se justifier par la relation suivante sur les colonnes $C_j$,
relation provenant de la deuxième différentielle
$\partial_{2,\delta+1}(\uP)$:
$$
-sC_1  -tC_2  -uC_3 + pC_4 + qC_5 + rC_6 = 0
$$
Le lecteur pourra réfléchir à la non-nullité des autres mineurs pleins. Pour
le jeu étalon généralisé $(pX, tY, fZ^2)$, quels sont
ceux qui \og survivent\fg?

Dans une introduction, c'est un peu délicat d'aller plus loin
concernant cette histoire de régularité car elle comprend plusieurs
volets.  Il y a par exemple un aspect global: les mineurs de $\Syl_d$
d'ordre le rang attendu~$s_d$ engendrent un idéal fidèle dès lors que
$\uP$ est régulière.  Mais on peut être intéressé par l'aspect
régulier d'\emph{un} mineur particulier ou bien de la forme linéaire
$\omega$ sous des conditions plus restrictives sur~$\uP$: couvrir le
jeu étalon généralisé ou bien générique. Ces divers points de vue ont
tous leur importance et seront examinés dans la suite.

\medskip

\noindent $\bullet$
Un troisième objectif va être de définir une famille emboîtée
$(\Jex_h)_{0 \leqslant h \leqslant n}$ d'idéaux monomiaux dits excédentaires.
Nous verrons ultérieurement qu'une certaine combinatoire, dépendant
uniquement du format $D$, relie ces idéaux aux différents termes du
complexe de Koszul de $\uX^D$.  Mais surtout, pour chaque $\uP$ de
format $D$, cette combinatoire permettra, dans le complexe de
Koszul de $\uP$, de relier certains mineurs 
de  la \textit{première} différentielle à des mineurs des \textit{autres}
différentielles, via des relations dites binomiales.

\medskip

\noindent $\bullet$
Le chapitre se termine par deux volets indépendants:
étude du poids de divers déterminants attachés à $\uP$ lorsque $\uP$ est générique
et profondeur de certaines familles de déterminants dits excédentaires.

\subsection{L'endomorphisme $W_1(\protect\uP)$ du $\bfA$-module $\Jex_1$
            et ses composantes homogènes}

Il s'agit, à tout système $\uP$ de format $D$, d'associer une famille
de matrices $\big(W_{1,d}(\uP)\big)_{d \geqslant 0}$. Leurs
déterminants sont déjà intervenus à plusieurs reprises en degré $d
\geqslant \delta+1$ comme mineurs pleins de~$\Syl_d$ dans le cas
d'école $n=2$ (cf.~\ref{DefW1deltaplus1nequal2}), et dans le chapitre
précédent (page~\pageref{pageAllusionW1dpXD}) où $\uP$
est le jeu étalon généralisé~$\pXD$.

Ici, nous allons définir $W_{1,d}$ comme un endomorphisme de
$\Jex_{1,d}$ via la fonction de sélection $\minDiv$ dont nous
rappelons la définition.

\begin{defn}[Sélection $\minDiv$]
\index{mécanisme de sélection}%
Il s'agit de la fonction $\{\hbox{monômes de $\bfA[\uX]$}\} \rightarrow \bbN$ définie par
$$
\minDiv : \ X^\alpha \ \longmapsto \ 
\left\{
\begin{array}{ll}
\min\big(\DivSeq(X^\alpha)\big) & \text{si $\#\DivSeq(X^\alpha)\geqslant 1$} \\
n+1 & \text{sinon} 
\end{array}
\right.
$$
\end{defn}

\label{NOTA05-minDiv}%
\index{mécanisme de sélection!$\minDiv$}%

Les $W_{1,d}(\uP)$ sont les composantes homogènes de l'endomorphisme
$W_1 = W_1(\uP)$ de l'idéal monomial $\Jex_1 = \langle X_1^{d_1},
\dots, X_n^{d_n}\rangle$ construit à l'aide de la
projection $\bfA$-linéaire $\pi_{\Jex_1} : \bfA[\uX]
\twoheadrightarrow \Jex_1$ relativement au supplémentaire monomial
$\Jex_{0\setminus 1}$
$$
\bfA[\uX] \ = \ \Jex_1\  \oplus \, \Jex_{0\setminus1}, \qquad\quad
\Jex_{0\setminus1} = \bigoplus_{\alpha \preccurlyeq \emouton} \! \bfA X^\alpha
$$

\begin{defn} \label{DefW1}
L'endomorphisme $W_1(\uP)$ du $\bfA$-module $\Jex_1$ est défini sur
la base monomiale par
$$
W_1(\uP) : \qquad
X^\alpha \ \longmapsto \ 
\pi_{\Jex_1} 
\bigg(\dfrac{X^\alpha}{X_i^{d_i}} P_i \bigg)
\qquad  \text{où $i = \minDiv(X^\alpha)$}
$$
\end{defn}

\label{NOTA05-piJ1}%
\label{NOTA05-W1d}%

Le déterminant de $W_{1,d} : \Jex_{1,d} \to \Jex_{1,d}$ est un mineur
de~$\Syl_d$ puisque $W_{1,d}$ a pour expression :
$$
X^\alpha 
\ \longmapsto \ 
\pi_{\Jex_{1, d}} \circ \Syl_d\biggl(
\dfrac{X^\alpha}{X_i^{d_i}}e_i\biggr)
\qquad  \text{où $i = \minDiv(X^\alpha)$}
$$

\begin {rmq}
Le lecteur peut se demander pourquoi avoir choisi le plus petit élément de
l'ensemble de divisibilité. C'est en fait arbitraire et on aurait pu par
exemple retenir le plus grand élément conduisant à la primitive $\maxDiv$ au lieu
de $\minDiv$. On verra ultérieurement (section \ref{SigmaVersion}) comment,
en modifiant l'ordre habituel de $\{1..n\}$, on peut rendre uniformes ces deux constructions 
et en proposer d'autres.
\end {rmq}

\label{NOTA05-maxDiv}%
\index{W@les endomorphismes!$W_1(\uP)$ du $\bfA$-module $\Jex_1$}%
\index{W@les endomorphismes!$W_{1,d}(\uP)$ de $\Jex_{1,d}$}%

\begin{prop}[Propriétés de $W_1$] \label{ProprieteW1} \leavevmode
 
\begin{enumerate}[\rm i)]
\item 
Pour le jeu étalon $\uX^D$, on a $W_1(\uX^D) = \Id$ de sorte que $\det W_{1,d}(\uX^D) = 1$.

\item 
Pour le jeu étalon généralisé $\pXD$, 
la matrice de $W_{1,d}(\pXD)$ est une matrice diagonale de $p_i$.
Chaque~$p_i$ y est représenté autant de fois que la dimension du 
$\bfA$-module $\Jex_{1,d}^{(i)}$ de base les monômes $X^\alpha \in \Jex_{1,d}$ 
vérifiant $\minDiv(X^\alpha) = i$. Ainsi:
$$
\det W_{1,d}(\pXD) = \prod_{X^\alpha \in \Jex_{1,d}} p_{\minDiv(X^\alpha)} =
\prod_{i=1}^n p_i^{\dim\Jex_{1,d}^{(i)}}
$$

\item 
Pour deux jeux $\uP$ et $\uQ$ de même format $D$, on a  
$W_1(\uP + \uQ) = W_1(\uP) + W_1(\uQ)$.

\item
Pour le jeu générique, le déterminant $\det W_{1,d}(\uP)$ est régulier.

\item 
Soit $\uP$ quelconque.
Pour $d \geqslant \delta+1$, le scalaire $\det W_{1,d}(\uP)$ appartient à
$\Ann(\bfB_d)$ donc à $\ElimIdeal$.
\end{enumerate}
\end{prop}

\label{NOTA05-Jex1di}%

\begin{proof}\leavevmode
  
Les points i), ii), iii) sont immédiats en revenant à la définition.

Pour iv), le polynôme $\det W_{1,d}(\uP)$ de $\bfA = \bfk[\indetsPi]$
est primitif par valeur puisque sa spécialisation en le jeu étalon
vaut~$1$, a fortiori régulier.

v) 
Raisonnons matriciellement. En notant $i_\alpha = \minDiv(X^\alpha)$, on a
l'égalité de matrices lignes:
$$
\Big[\ 
\big((X^\alpha/X_{i_\alpha}^{d_{i_\alpha}})\,P_{i_\alpha} \big)_{X^\alpha \in \Jex_{1,d}}
\  \Big]
\ \ = \ \ 
\Big[\ 
\big(X^\alpha\big)_{X^\alpha \in \Jex_{1,d}} 
\ \Big]
\ W_{1,d}
$$
En multipliant à droite par la transposée de la comatrice de $W_{1,d}$,
on obtient $X^\alpha \det W_{1,d} \in \langle \uP \rangle_d$ 
pour tout $X^\alpha \in \Jex_{1,d}$.
Or $d \geqslant \delta+1$, donc $\Jex_{1,d} = \bfA[\uX]_d$. Et l'appartenance 
précédente traduit donc que $\det W_{1,d}$ est dans $\Ann(\bfB_d)$.
\end{proof}

\begin{rmq}
Le point iii), certes banal, est bien utile, cf par exemple son
application au jeu circulaire $\uQ = \pXD + \uR$ dans le
chapitre~\ref{ChapJeuCirculaire}.  Chez Demazure
\cite[p. 20]{Demazure1}, c'est l'égalité $W_1(t\uX^D + \uP) = t \,
\Id + W_1(\uP)$ qui sera exploitée, où $t$ est \textit{une}
indéterminée.
\end {rmq}

\begin{rmq}
En terrain générique, pour $d \leqslant \delta$, nous montrerons dans
le chapitre \ref{ChapBgenerique} (cf en particulier
\ref{SectPsatRegBdeltaRegII}) que $\det W_{1,d}$ est régulier
modulo $\ElimIdeal$. A fortiori, $\det W_{1,d} \notin \ElimIdeal$; ce
dernier point est prouvé par J.-P. Jouanolou dans \cite[Remarque 3.9.4.5, page 129]{J7} en
invoquant le degré d'homogénéité du déterminant en les coefficients de
$P_i$ et en le comparant à celui du résultant, notion qui est en place chez
cet auteur depuis son article~\cite{J3}. Dans le chapitre
\ref{ChapBgenerique}, notre approche de la régularité modulo
$\ElimIdeal$ est radicalement différente et repose sur la combinatoire
du jeu circulaire (cf chapitre~\ref {ChapJeuCirculaire}).
\end{rmq}

\subsection{L'endomorphisme $W_{\calM}(\protect\uP)$ pour un sous-module
            monomial $\calM$ de $\Jex_1$}

Pour un sous-$\bfA$-module monomial $\calM$ de $\Jex_1$, on peut
considérer l'endomorphisme~$W_\calM$ induit-projeté sur~$\calM$ de
l'endomorphisme~$W_1$.  Concrètement, pour obtenir l'image de
$X^\alpha \in \calM$ par $W_\calM$, il suffit de reprendre
l'expression donnée en~\ref{DefW1} et garder uniquement les monômes
appartenant à~$\calM$.

\begin{defn} \label{DefWM}

On note $\pi_\calM : \bfA[\uX] \to \calM$ la projection sur $\calM$
relativement à la somme directe $\bfA[\uX] = \calM \oplus \overline
\calM$ où $\overline \calM$ est le supplémentaire monomial de $\calM$
dans $\bfA[\uX]$. L'endomorphisme $W_\calM(\uP)$ du $\bfA$-module
$\calM$ est défini sur la base monomiale par
$$
W_\calM(\uP) : \qquad
X^\alpha \ \longmapsto \ 
\pi_{\calM} 
\bigg(\dfrac{X^\alpha}{X_i^{d_i}} P_i \bigg)
\qquad  \text{où $i = \minDiv(X^\alpha)$}
$$
On a bien sûr $W_{\Jex_1} = W_1$ et $W_{\Jex_{1,d}} = W_{1,d}$.
En prenant $\calM \subset \Jex_{1,d}$, 
le déterminant de $W_{\calM}(\uP)$ est un mineur de l'application de Sylvester $\Syl_d(\uP)$.
\end{defn}

\label{NOTA05-piM}%
\label{NOTA05-WM}%
\index{W@les endomorphismes!$W_\calM(\uP)$ du sous-module monomial $\calM\subset\Jex_1$}%

\noindent
\textbf{Exemple : $D = (3,1,2)$ et $d=4$}
$$
\setlength{\tabcolsep}{2pt}
\left\{\begin{tabular}{rcp{15cm}} 
$P_{1}$ & $=$ & $a_{1}X^{3} + a_{2}X^{2}Y + a_{3}X^{2}Z + a_{4}XY^{2} + a_{5}XYZ + a_{6}XZ^{2} + a_{7}Y^{3} + a_{8}Y^{2}Z + a_{9}YZ^{2} 
+ a_{10}Z^{3}$
\\ [0.1cm] 
$P_{2}$ & $=$ & $b_{1}X + b_{2}Y + b_{3}Z$
\\ [0.1cm] 
$P_{3}$ & $=$ & $c_{1}X^{2} + c_{2}XY + c_{3}XZ + c_{4}Y^{2} + c_{5}YZ + c_{6}Z^{2}$
\\ [0.1cm] 
\end{tabular}
\right.
$$
Voici la matrice de Sylvester $\Syl_d(\uP)$:
$$
\Syl_{d} \ = \ 
\NorthEastBordermatrix{
\Veti{X\,e_{1}} &\Veti{Y\,e_{1}} &\Veti{Z\,e_{1}} &\Veti{X^{3}\,e_{2}} &\Veti{X^{2}Y\,e_{2}} &\Veti{X^{2}Z\,e_{2}} &\Veti{XY^{2}\,e_{2}} &\Veti{XYZ\,e_{2}} &\Veti{XZ^{2}\,e_{2}} &\Veti{Y^{3}\,e_{2}} &\Veti{Y^{2}Z\,e_{2}} &\Veti{YZ^{2}\,e_{2}} &\Veti{Z^{3}\,e_{2}} &\Veti{X^{2}\,e_{3}} &\Veti{XY\,e_{3}} &\Veti{XZ\,e_{3}} &\Veti{Y^{2}\,e_{3}} &\Veti{YZ\,e_{3}} &\Veti{Z^{2}\,e_{3}} &\\
a_{1} &. &. &b_{1} &. &. &. &. &. &. &. &. &. &c_{1} &. &. &. &. &. &\Heti{X^{4}} \\
a_{2} &a_{1} &. &b_{2} &b_{1} &. &. &. &. &. &. &. &. &c_{2} &c_{1} &. &. &. &. &\Heti{X^{3}Y} \\
a_{3} &. &a_{1} &b_{3} &. &b_{1} &. &. &. &. &. &. &. &c_{3} &. &c_{1} &. &. &. &\Heti{X^{3}Z} \\
a_{4} &a_{2} &. &. &b_{2} &. &b_{1} &. &. &. &. &. &. &c_{4} &c_{2} &. &c_{1} &. &. &\Heti{X^{2}Y^{2}} \\
a_{5} &a_{3} &a_{2} &. &b_{3} &b_{2} &. &b_{1} &. &. &. &. &. &c_{5} &c_{3} &c_{2} &. &c_{1} &. &\Heti{X^{2}YZ} \\
a_{6} &. &a_{3} &. &. &b_{3} &. &. &b_{1} &. &. &. &. &c_{6} &. &c_{3} &. &. &c_{1} &\Heti{X^{2}Z^{2}} \\
a_{7} &a_{4} &. &. &. &. &b_{2} &. &. &b_{1} &. &. &. &. &c_{4} &. &c_{2} &. &. &\Heti{XY^{3}} \\
a_{8} &a_{5} &a_{4} &. &. &. &b_{3} &b_{2} &. &. &b_{1} &. &. &. &c_{5} &c_{4} &c_{3} &c_{2} &. &\Heti{XY^{2}Z} \\
a_{9} &a_{6} &a_{5} &. &. &. &. &b_{3} &b_{2} &. &. &b_{1} &. &. &c_{6} &c_{5} &. &c_{3} &c_{2} &\Heti{XYZ^{2}} \\
a_{10} &. &a_{6} &. &. &. &. &. &b_{3} &. &. &. &b_{1} &. &. &c_{6} &. &. &c_{3} &\Heti{XZ^{3}} \\
. &a_{7} &. &. &. &. &. &. &. &b_{2} &. &. &. &. &. &. &c_{4} &. &. &\Heti{Y^{4}} \\
. &a_{8} &a_{7} &. &. &. &. &. &. &b_{3} &b_{2} &. &. &. &. &. &c_{5} &c_{4} &. &\Heti{Y^{3}Z} \\
. &a_{9} &a_{8} &. &. &. &. &. &. &. &b_{3} &b_{2} &. &. &. &. &c_{6} &c_{5} &c_{4} &\Heti{Y^{2}Z^{2}} \\
. &a_{10} &a_{9} &. &. &. &. &. &. &. &. &b_{3} &b_{2} &. &. &. &. &c_{6} &c_{5} &\Heti{YZ^{3}} \\
. &. &a_{10} &. &. &. &. &. &. &. &. &. &b_{3} &. &. &. &. &. &c_{6} &\Heti{Z^{4}} \\
}
$$
Soit $\calM$ le sous-module de $\Jex_{1,d}$ de base les 10 monômes
divisibles par $Z$. Déterminons l'image de $XYZ^2$ par~$W_\calM$; le
minimum de $\DivSeq(XYZ^2) = \{2,3\}$ étant 2, la colonne
$\frac{XYZ^2}{Y} e_2 = XZ^2 e_2$ de $\Syl_d$ est sélectionnée et
l'image en question en est la projection sur $\calM$.
$$
W_\calM \ = \ 
\EastBordermatrix{
a_{1} & b_{1} & c_{3} & . & . & c_{1} & . & . & . & . & \Heti{X^{3}Z} \\ 
a_{2} & b_{2} & c_{5} & b_{1} & . & c_{2} & . & . & . & . & \Heti{X^{2}YZ} \\ 
a_{3} & b_{3} & c_{6} & . & b_{1} & c_{3} & . & . & . & c_{1} & \Heti{X^{2}Z^{2}} \\ 
a_{4} & . & . & b_{2} & . & c_{4} & b_{1} & . & . & . & \Heti{XY^{2}Z} \\ 
a_{5} & . & . & b_{3} & b_{2} & c_{5} & . & b_{1} & . & c_{2} & \Heti{XYZ^{2}} \\ 
a_{6} & . & . & . & b_{3} & c_{6} & . & . & b_{1} & c_{3} & \Heti{XZ^{3}} \\ 
a_{7} & . & . & . & . & . & b_{2} & . & . & . & \Heti{Y^{3}Z} \\ 
a_{8} & . & . & . & . & . & b_{3} & b_{2} & . & c_{4} & \Heti{Y^{2}Z^{2}} \\ 
a_{9} & . & . & . & . & . & . & b_{3} & b_{2} & c_{5} & \Heti{YZ^{3}} \\ 
a_{10} & . & . & . & . & . & . & . & b_{3} & c_{6} & \Heti{Z^{4}} \\ 
}
$$

\subsubsection*{De l'importance de la $\bfA$-base monomiale de $\bfA[\uX]$ pour la notion d'induit-projeté}

De manière générale, pour un endomorphisme $u$ d'un module $G = E \oplus F$, l'endomorphisme 
induit-projeté de $u$ sur $E$ est $\pi_{E \parallel F} \circ u \circ \iota_E$ où  
$\pi_{E \parallel F}$ désigne la projection de $E$ parallèlement à $F$ et $\iota_E$ l'injection 
canonique de $E$ dans $G$.
Dans le cas où $G$ est un module libre basé, 
ceci nous permet de parler de \og la \fg{} matrice de $u$. 
Soit $\calB$ une partie de la base de $G$ et 
$E$ le sous-module de base $\calB$.
On a alors $G = E \oplus \overline E$ où $\overline E$ est le sous-module de base
le complémentaire de $\calB$ dans la base de $G$.
L'induit-projeté $u_E$ de $u$ sur~$E$ est alors défini par 
$$
u_E : \ 
x \mapsto \sum_{b \in \calB} 
\big[ u(x) \big]_b \, b
\qquad \quad 
\text{où $\big[ u(x) \big]_b$ désigne la composante de $u(x)$ sur le vecteur $b$}
$$
Matriciellement, on obtient la matrice de $u_E$ à partir de celle de
$u$ en supprimant ligne et colonne correspondant aux vecteurs non
dans $\calB$.  Cette construction conserve l'aspect
triangulaire: si la matrice de $u$ est triangulaire, alors la matrice
de $u_E$ l'est également.

\index{induit-projeté}%
%
%

\subsubsection*{Un exemple élémentaire pour illustrer l'utilisation des matrices $W_\sbullet$ et des induit-projetés}

Soit $\calN$ un sous-module monomial de $\Jex_{1,d}$ et $I \subset \{1..n\}$
une partie quelconque. Partageons l'ensemble des monômes $X^\alpha \in \calN$
en deux paquets selon $(\minDiv, I)$ de manière à obtenir
deux sous-modules monomiaux $\calN', \calN''$ de  $\calN$ supplémentaires
l'un de l'autre:
$$
\calN' : \minDiv(X^\alpha) \notin I, \qquad
\calN'' : \minDiv(X^\alpha) \in I, \qquad\qquad
\calN = \calN' \oplus \calN''
$$
Pour un jeu $\uP$ quelconque, on a, de manière banale, la structure par blocs suivante:
$$
W_\calN(\uP) 
\ = \ 
\EastBordermatrix{
\,W_{\calN'}(\uP)     & \star   & \calN'
\\ \noalign {\medskip}
\star          & W_{\calN''}(\uP) & \calN'' 
}
$$
Par définition de $\calN'$ et $\calN''$ vis à vis de $I$, les coefficients des $(P_i)_{i \in I}$ ne
figurent que dans la deuxième colonne-bloc, colonne-bloc elle-même constituée uniquement de coefficients
des $(P_i)_{i\in I}$ (et de 0). 
Définissons un nouveau système $\uQ = (Q_1, \dots, Q_n)$ de la manière suivante:
$$
Q_i = \begin {cases}
P_i      &\text {si } i \notin I\\
X_i^{d_i} &\text {si } i \in I\\
\end {cases}
$$
Alors $W_\calN(\uQ)$ a la structure suivante:
$$
W_\calN(\uQ) 
\ = \ 
\EastBordermatrix{
\,W_{\calN'}(\uQ)  & 0   & \calN'
\\ \noalign {\medskip}
\star            &\Id & \calN'' 
}
$$
Et de plus $W_{\calN'}(\uQ) = W_{\calN'}(\uP)$.
On peut donc énoncer le lemme (banal) suivant qui sera utilisé dans le chapitre~\ref{ChapWW}
lors de la mise au point de formules de récurrence pour les déterminants $\det W_\sbullet$.

\begin{lem} \label{BiPartitionMinDivLemma}
Dans le contexte ci-dessus, on dispose des égalités déterminantales:
$$
\det W_\calN(\uQ) \ = \  \det W_{\calN'}(\uQ) \ = \  \det W_{\calN'}(\uP)
$$
\end{lem}

\subsection{Les idéaux monomiaux $\Jex_h$ et les endomorphismes $W_{h,d}(\protect\uP)$ de $\Jex_{h,d}$}

Parmi les sous-modules monomiaux $\calM$ de $\Jex_1$, il y a l'idéal
monomial $\Jex_2 = \langle X_i^{d_i}X_j^{d_j} \mid i \neq j \rangle$
qui joue un rôle privilégié car il intervient dans la formule de Macaulay
ci-dessous prouvée ultérieurement (cf. section~\ref{sousSectionMacaulayRecurrence}).
Pour $\calM=\Jex_2$, l'endomorphisme $W_\calM$ de $\Jex_2$ est noté $W_2$ et sa composante
homogène, endomorphisme de $\Jex_{2,d}$, est notée $W_{2,d}$.
$$
\forall\, d \geqslant \delta+1, \qquad \Res(\uP) = \frac {\det W_{1,d}(\uP)}{\det W_{2,d}(\uP)}  
\leqno (\hbox {Macaulay})
$$

\index{W@les endomorphismes!$W_2(\uP)$ du $\bfA$-module $\Jex_2$}%
\index{W@les endomorphismes!$W_{2,d}(\uP)$ de $\Jex_{2,d}$}%
\label{NOTA05-W2d}%

\medskip
Les déterminants $\det W_{1,d}, \det W_{2,d}$ vérifient certaines
propriétés de divisibilité abordées par la suite. Celles concernant
$W_1$ font l'objet du théorème~\ref{DetW1dDiviseDetW1dplus1}, d'autres
du chapitre~\ref{ChapWW}. En ce qui concerne $W_2$, signalons le
résultat suivant, qui sera prouvé en~\ref{DivisibiliteEntreW2}:

\begin{theo}[Divisibilité entre déterminants de type $W_2$]
Pour une suite $\uP$ quelconque et pour tout $d$, le déterminant 
$\det W_{2,d}$ divise $\det W_{2, d+1}$.
\end{theo}

\medskip

Nous allons maintenant introduire une famille emboîtée 
$(\Jex_h)_{0 \leqslant h \leqslant n}$ d'idéaux monomiaux de 
$\bfA[\uX]$ dont la définition,
calquée sur celle de $\Jex_1$ et $\Jex_2$, ne dépend que du format $D =
(d_1, \dots, d_n)$. Ce qui va nous permettre, pour tout système $\uP$
de format $D$, de définir un endomorphisme $W_h(\uP)$ analogue
aux $W_1(\uP)$ et $W_2(\uP)$.

\medskip
S'il est aisé de motiver l'introduction de $W_1, W_2$ via la formule de Macaulay,
c'est un peu plus délicat d'expliquer dans ce chapitre le rôle de la famille
$W_h(\uP)$. On peut tout de même évoquer le fait que les $\Jex_{h,d}$ \og figurent\fg{} dans
le terme $\rmK_{k,d}$, composante homogène de degré~$d$ du  complexe de Koszul du
jeu étalon généralisé $(X_1^{d_1}, \dots, X_n^{d_n})$, sous la forme d'une décomposition
explicite nommée \og décomposition binomiale\fg:
$$
\rmK_{k,d}
\ \simeq \ 
\bigoplus_{h=k}^n \Jex_{h,d}^{\oplus e'_{k,h}} \qquad \hbox{où}\
  e'_{k,h} = \binom{h-1}{h-k}
$$
Cette décomposition, mise au point dans un chapitre ultérieur
(cf. la proposition~\ref{DecompositionKkdParJhd})
permettra de chasser dans les mineurs des différentielles du
complexe de Koszul et de faire apparaître des endomorphismes
\og de Cayley\fg{} $B_k$
liés aux endomorphismes \og de Macaulay\fg{} $W_h$  par des relations très surprenantes
qualifiées de~\og relations binomiales \fg.
Nous nous posons d'ailleurs la question: ces relations ont-elles un rapport avec la phrase
de Demazure en 1984 dans~\cite[p.~2]{Demazure1} ?

\begin{quote}
\textit{Il est peut-être regrettable, d'un certain point de vue, que certaines démonstrations
ne soient pas purement calculatoires; cela montre qu'il reste encore bien des choses
à comprendre en cette affaire.}
\end{quote}

\bigskip

Les idéaux monomiaux $\Jex_h$ de $\bfA[\uX]$ dépendent du
format~$D = (d_1, \dots, d_n)$ de degrés, mais nous ne faisons pas
intervenir cette dépendance dans les notations.  Ils généralisent les
idéaux engendrés par les monômes repus et dodus de Jouanolou
(cf. \cite[section 3.9.1]{J7}).

\begin{defn}
\label{defJh}  
Pour $h \in \bbN$, l'idéal des monômes $h$-excédentaires est l'idéal monomial:
$$
\Jex_{h} 
\ = \ 
\langle X^{D(I)} \mid \# I = h \rangle 
\qquad 
\text{où \ $X^{D(I)} \,= \, \prod_{i \in I} X_i^{d_i}$}
$$
On a $\Jex_h = 0$ pour $h \geqslant n+1$ et 
$\Jex_n = \langle X^D \rangle$ où $X^D = \prod_{i=1}^n X_i^{d_i}$.
On a également les inclusions:
$$
0 \,=\, \Jex_{n+1} \ \subset \ 
\Jex_{n} \ \subset \ 
\cdots\ \subset \ 
\Jex_{1} \subset \  \Jex_0 \, = \, \bfA[\uX]
$$
\end{defn}

\index{idéal!monomial excédentaire!$\Jex_h$}%
\label{NOTA05-Jexh}%
\label{NOTA05-XDI}%
%
%

\begin{prop}[Nullité de $\Jex_{h,d}$] \label{NulliteJhd}
Pour une partie $I \subset \{1..n\}$, posons $d_I = \sum\limits_{i\in I}d_i$ et 
$d_{\min}(\Jex_h) = \min\limits_{\#I=h} d_I$.
L'idéal~$\Jex_h$ contient un monôme de degré~$d$ si et seulement si 
$d \geqslant d_{\min}(\Jex_h)$. Pris à rebours :
$$
\Jex_{h,d} = 0 
\quad \Leftrightarrow \quad 
d < d_{\min}(\Jex_h)
$$
En particulier, on a $\Jex_{n,d} = 0 \Leftrightarrow d < \delta + n$.
\end{prop}

\begin{proof}
Si $d < d_{\rm min}(\Jex_h)$, il est clair que $\Jex_{h,d} = 0$.  Dans
l'autre sens, si $d \geqslant d_{\min}(\Jex_h)$, il existe $I$ tel que $\#I
= h$ et $d \geqslant d_I$ ; alors pour $|\gamma| = d-d_I$, le monôme
$X^\gamma X^{D(I)} \overset{\rm def}{=} X^\gamma \prod_{i \in I}
X_i^{d_i}$ est dans $\Jex_{h,d}$.  Le cas particulier de $h=n$
provient de $d_{\min}(\Jex_n) = \sum_i d_i = \delta+n$.
\end{proof}

Le lien entre l'idéal excédentaire $\Jex_{k}$ et le terme $\rmK_{k}$ du complexe
de Koszul de $\uX^D$ sera approfondi plus tard (chapitre~\ref{ChapBW}).
En voici un aspect vraiment élémentaire.

\begin{prop} \label{NulliteJkdKkd}
On a l'équivalence $\Jex_{k,d} = 0 \, \Leftrightarrow \, \rmK_{k,d} = 0$. 
En particulier, on retrouve l'implication $\rmK_{k,d} = 0 \Rightarrow \rmK_{k+1,d} =0$ du point iii) de
\ref{dminKk}.
\end{prop}

\begin{proof}
Une première justification consiste à utiliser~\ref{dminKk} 
qui fournit $d_{\min}(\rmK_k)= \min\limits_{\#I =k} d_I = d_{\min}(\Jex_k)$.

On peut également fournir des \og monômes témoins\fg.
Si $X^\alpha e_I \in \rmK_{k,d}$, alors $X^\alpha X^{D(I)}$ est dans $\Jex_{k,d}$.
Dans l'autre sens, soit $X^\alpha \in \Jex_{k,d}$ que l'on écrit $X^\alpha = X^{\gamma} X^{D(I)}$ avec $\#I = k$ ;
alors $X^\gamma e_I$ est dans $\rmK_{k,d}$.

\smallskip
Le \og En particulier \fg{} se déduit de l'inclusion $\Jex_{k+1,d} \subset \Jex_{k,d}$.
\end{proof}

\begin{defn}[L'endomorphisme $W_h(\uP)$ et le déterminant $\det W_{h,d}(\uP)$]\leavevmode
\label{DefWh}

On note $W_h$ au lieu de $W_{\Jex_h}$ l'endomorphisme du $\bfA$-module $\Jex_h$
et $W_{h,d}$ sa composante homogène de degré $d$, encore égale à $W_{\Jex_{h,d}}$.
Chaque $\det W_{h,d}(\uP)$ est un mineur de l'application de Sylvester~$\Syl_d(\uP)$.
\end{defn}

\label{NOTA05-Wh}%
\index{W@les endomorphismes!$W_{h,d}(\uP)$ de $\Jex_{h,d}$}%

\medskip

Voici un résultat élémentaire concernant le comportement de
l'endomorphisme $W_h(\uP)$ vis-à-vis du jeu de polynômes $\uP$.

\begin{prop} \label{ProprieteWh}
\leavevmode
\begin{enumerate}[\rm i)]
\item 
Pour le jeu étalon $\uX^D$, on a $W_h(\uX^D) = \Id$ de sorte que $\det W_{h,d}(\uX^D) = 1$.

\item
Pour le jeu générique, le déterminant $\det W_{h,d}(\uP)$ est régulier.

\item 
Pour deux jeux $\uP$ et $\uQ$ de même format $D$, on a  
$W_h(\uP + \uQ) = W_h(\uP) + W_h(\uQ)$.
 
En particulier, pour une indéterminée $t$, on a 
$W_h(\uP + t\uX^D) = W_h(\uP) + t \, \Id$. 

\item
L'endomorphisme $W_h(\uP)$ ne dépend pas des $h-1$ derniers polynômes de la suite~$\uP$.
\end{enumerate}
\end{prop}

\begin{proof}
Les points i) et iii) sont immédiats en revenant à la définition.
Pour ii), le polynôme $\det W_{h,d}(\uP)$ 
de $\bfA = \bfk[\indetsPi]$ est primitif par valeur 
(en effet, pour le jeu étalon, il vaut~$1$).

\smallskip\noindent
Le point iv) : soit $X^\alpha \in \Jex_h$, c'est-à-dire $\# \DivSeq(X^\alpha) \geqslant h$.
L'image de $X^\alpha$ par $W_h(\uP)$
dépend uniquement du polynôme $P_i$ où $i=\minDiv(X^\alpha)$.
Soit $E$ l'ensemble des $h-1$ derniers indices de $\{1..n\}$.
Comme la partie $\DivSeq(X^\alpha)$ est de cardinal $\geqslant h$, elle
rencontre $\overline E := \{1..n\} \setminus E$ de cardinal $n-h+1$; donc $i \in \overline E$.
\end{proof}

\subsection{La forme linéaire $\omega_\protect\uP : \bfA[\protect\uX]_\delta\rightarrow\bfA$ en degré $\delta$
(pilotée par $\minDiv$)}
\label{sectionFormeLineaireOmega}

\index{forme linéaire!$\omega=\omega_\uP$ (pilotée par $\minDiv$)}%

La forme linéaire $\omega = \omega_{\uP} : \bfA[\uX]_\delta \to \bfA$
est déjà intervenue à plusieurs reprises: dans l'étude du cas d'école
$n=2$ cf.~\ref{Defomegan=2} et celle du jeu étalon généralisé
en~\ref{omegaCasParticulier},
page~\pageref{MacRaeJeuEtalonDelta}.  Ici nous choisissons de la
définir par l'intermédiaire d'une construction notée $\Omega$. Cette
construction repose sur le fait qu'en degré~$\delta$, le monôme
\MoutonNoir{}~$X^\emouton$ est le seul à ne pas appartenir à $\Jex_{1} =
\langle X_1^{d_1}, \dots, X_n^{d_n} \rangle$, ce qui conduit à
$\bfA[\uX]_\delta = \Jex_{1,\delta} \, \oplus \, \bfA X^{\emouton}$.
Par ailleurs, $\Omega$ privilégie le mécanisme de sélection $\minDiv$.

\begin{defn}\label{DefEndoOmega}
Pour $F \in \bfA[\uX]_\delta$, l'endomorphisme $\Omega(F)$ est défini par :
$$
\Omega(F) : \ 
\begin{array}[t]{rcl}
\bfA[\uX]_\delta & \longrightarrow & \bfA[\uX]_\delta \\ [0.4cm]
X^\alpha & \longmapsto & 
\left\{
\begin{array}{ll}
\dfrac{X^\alpha}{X_i^{d_i}} P_i & 
\text{si $X^\alpha \in \Jex_{1,\delta}$ où $i = \minDiv(X^\alpha)$} \\ [0.6cm]
F & 
\text{si $X^\alpha = X^\emouton$} 
\end{array}
\right.
\end{array}
$$
Visuellement :
$$
\begin{tikzpicture} 
\draw [thick] (-0.1, -0.3) -- (-0.4, -0.3) -- (-0.4, 6.3) -- (-0.1, 6.3) ;
\draw [thick] (6+0.1, -0.3) -- (6+0.4, -0.3) -- (6+0.4, 6.3) -- (6+0.1, 6.3) ;
\draw [dashed, rounded corners=4pt, fill=gray!30] (3.8, -0.3) rectangle (4.2, 6.3) ;
\draw (4, 6.4) node[above] {$F$} ;
\draw [dotted] (4, 2) -- (6.6, 2) node[right] {\mouton} ;
\draw [dashed, rounded corners=4pt] (1.8,-0.3) rectangle (2.2, 6.3) ;
\draw (2,6.3) node[above] {\small $\dfrac{X^\alpha}{X_i^{d_i}} P_i$} ;
\draw [dotted] (2,4) -- (6.6, 4) node[right] {$X^\alpha$} ;
\foreach \r in {0,1,...,6} \draw (\r, 6-\r) node {\tiny $\bullet$} ;
\path (0,0) -- (0, 6) node[midway, left] {$\Omega(F) \ = \ \quad $} ;
\end{tikzpicture}
$$
On définit alors la forme linéaire $\omega$ grâce au déterminant :
$$
\omega : 
\ 
\bfA[\uX]_\delta \ \longrightarrow \ \bfA, 
\qquad 
F \ \longmapsto \ \det \big( \Omega(F) \big)
$$
\end{defn}

\label{NOTA05-Omega}%
\label{NOTA05-omega}%
%
%

\index{O@Omega-constructeur à valeurs dans!$\End_\bfA(\bfA[\uX]_\delta)$!1@$\Omega=\Omega_\uP$}%

\medskip

\noindent
Ainsi, pour $D = (3,1,2)$ de $\delta = 3$ et le système $\uP$ :

\medskip
{%
\setlength{\tabcolsep}{2pt}
\noindent
$
\left\{
\begin{tabular}{@{} rcp{15cm}} 
$P_{1}$ & $=$ & $a_{1}X^{3} + a_{2}X^{2}Y + a_{3}X^{2}Z + a_{4}XY^{2} +
  a_{5}XYZ + a_{6}XZ^{2} + 
  a_{7}Y^{3} + a_{8}Y^{2}Z + a_{9}YZ^{2} + a_{10}Z^{3}$\\ [0.1cm] 
$P_{2}$ & $=$ & $b_{1}X + b_{2}Y + b_{3}Z$\\ [0.1cm] 
$P_{3}$ & $=$ & $c_{1}X^{2} + c_{2}XY + c_{3}XZ + c_{4}Y^{2} +
c_{5}YZ + c_{6}Z^{2}$\\ [0.1cm] 
\end{tabular} 
\right.
$
}

\smallskip
\noindent
on a :
$$
\omega(\sbullet) \ = \ 
\EastBorderdet{
a_{1} & b_{1} & \sbullet & . & . & c_{1} & . & . & . & . & \Heti{X^{3}} \\ 
a_{2} & b_{2} & \sbullet & b_{1} & . & c_{2} & . & . & . & . & \Heti{X^{2}Y} \\ 
a_{3} & b_{3} & \sbullet & . & b_{1} & c_{3} & . & . & . & c_{1} & \Heti{X^{2}Z} \ \mouton \\ 
a_{4} & . & \sbullet & b_{2} & . & c_{4} & b_{1} & . & . & . & \Heti{XY^{2}} \\ 
a_{5} & . & \sbullet & b_{3} & b_{2} & c_{5} & . & b_{1} & . & c_{2} & \Heti{XYZ} \\ 
a_{6} & . & \sbullet & . & b_{3} & c_{6} & . & . & b_{1} & c_{3} & \Heti{XZ^{2}} \\ 
a_{7} & . & \sbullet & . & . & . & b_{2} & . & . & . & \Heti{Y^{3}} \\ 
a_{8} & . & \sbullet & . & . & . & b_{3} & b_{2} & . & c_{4} & \Heti{Y^{2}Z} \\ 
a_{9} & . & \sbullet & . & . & . & . & b_{3} & b_{2} & c_{5} & \Heti{YZ^{2}} \\ 
a_{10} & . & \sbullet & . & . & . & . & . & b_{3} & c_{6} & \Heti{Z^{3}} \\ 
}
$$
\label{omegaDessin}%

Par construction, cette forme linéaire $\omega$ 
appartient au sous-module $\DVect_{s_\delta}(\Syl_\delta) \subset \bfA[\uX]_\delta^\star$
qui est intervenu en \ref {NOTA04-DvectSyldelta}
et qui sera étudié dans un cadre plus général en~\ref{DefDVect}.
\index{module!déterminantiel $\DVect_{s_\delta}(\Syl_\delta)$}%
Mis à part la colonne-argument en indice \MoutonNoir{}, les colonnes
de $\Omega(F)$ sont celles de $\Syl_\delta$ indexées par les
$\frac{X^\alpha}{X_i^{d_i}} e_i$ où $i = \minDiv(X^\alpha)$.
Matriciellement, 
$\Omega(F)$ agit sur la base monomiale de $\bfA[\uX]_{\delta}$ de la manière suivante :
$$
\Big[\ 
\Big(\frac{X^\alpha}{X_i^{d_i}}\,P_i \Big)_{X^\alpha \in \Jex_{1,\delta}}
\ \mid \ 
F
\  \Big]
\ \ = \ \ 
\Big[\ 
(X^\alpha)_{X^\alpha \in \Jex_{1,\delta}} \ \mid \ X^{\emouton}
\ \Big]
\ \Omega(F)
$$
En particulier, pour le jeu étalon, on a $\Omega(X^{\emouton}) = \Id$.

\medskip

Cette description matricielle peut s'abstraire à l'aide de l'algèbre
extérieure.  En~\ref{soussection-d-egal-delta},
page~\pageref{soussection-d-egal-delta}, nous avons décrit comment,
à partir d'une application linéaire $u : E \to F$, construire sur $F$
une forme linéaire déterminantale sans ambiguïté de signe, à condition
de disposer d'une base $\calB_F$ de $F$, d'un élément distingué $f$ de
cette base et d'une application~$\iota : \calB_F \setminus \{f\} \to
E$.  On applique cette construction à l'application
linéaire~$\Syl_\delta : \rmK_{1,\delta} \to \bfA[\uX]_\delta$, à la
base monomiale de $\bfA[\uX]_\delta$ et au \MoutonNoir{} pour produire
la forme linéaire $\omega$.  Ici $\iota$ est l'application gouvernée
par le mécanisme de sélection $\minDiv$
$$
\iota : 
\begin{array}[t]{rcl}
\{\hbox {monômes de $\Jex_{1,\delta}$}\} & \longrightarrow & \rmK_{1,\delta} \\ [0.2cm]
X^\alpha & \longmapsto & \dfrac{X^\alpha}{X_i^{d_i}}\, e_i  
\hbox { avec $i = \minDiv(\alpha)$}
\end{array}
$$
En particulier, pour $|\gamma| = \delta$, 
le scalaire $\omega(X^\gamma)$ est un mineur 
d'ordre $s_\delta$ de $\Syl_\delta$ sans ambiguïté de signe.

\medskip
La forme linéaire $\omega$ s'obtient donc en ordonnant la base de
$\bfA[\uX]_\delta$ relativement à la décomposition $\Jex_{1,\delta}
\oplus \bfA X^\emouton$ et en posant $\bff = \!\!\!\!
\bigwedge\limits_{X^\alpha \in \Jex_{1,\delta}} \!\!\!\!  X^\alpha \,
\wedge \, X^{\emouton}$:
$$
\omega : \quad
F \ \longmapsto\ 
\Big[
\bigwedge_{X^\alpha \in \Jex_{1,\delta}}
\!\!\!\!
\Syl_\delta\big(\iota(X^\alpha) \big)
\ \wedge \ 
F
\Big]_{\bff}
$$

\begin{prop}[Propriétés de $\omega$]
\label{omegaProprietes}
\leavevmode

\begin{enumerate}[\rm i)]

\item 
La forme linéaire $\omega$ est de Cramer:
$$
\forall\, 
F, \, G \in \bfA[\uX]_\delta, \qquad 
\omega(F)\, G - \omega(G)\, F \, \in \, \uPdelta
$$

\item 
Si $F \in \uPdelta$ alors $\omega(F) = 0$. L'inclusion $\uPdelta \subset \Ker\omega$
s'écrit encore $\omega \in \Ker \transpose{\Syl_\delta}$.

\item 
Pour $F$ et $G \in \bfA[\uX]_\delta$ 
tels que $F \equiv G \bmod \uPdelta$, on a $\omega(F) = \omega(G)$.
En particulier, pour deux déterminants bezoutiens
$\nabla$ et $\nabla'$, on a $\omega(\nabla) = \omega(\nabla')$.

\item 
On a $\omega(X^\emouton) = \det W_{1,\delta}$.
En particulier, pour le jeu étalon, on a $\omega(\nabla) = 1$.

\item 
Pour $F \in \bfA[\uX]_\delta$, on a 
$$
\forall |\alpha | = \delta, \qquad 
X^\alpha \, \omega(F) 
\ \in \ 
\uPdelta \, + \, \bfA \, F
$$

\item 
Pour tout $F \in \uPsat_\delta$, on a $\omega(F) \in \ElimIdeal$.
En particulier, $\omega(\nabla) \in \ElimIdeal$.
Et ce scalaire~$\omega(\nabla)$ est régulier en générique.
\end{enumerate}
\end{prop}

\index{propriété Cramer!2@d'une forme linéaire}%

\begin{proof}
i) Appliquer \ref{CramerSymetrie}.

ii) Il suffit de faire la preuve pour la suite générique $\uP$.
Soit $F \in \uPdelta$.
On a donc $\omega(F) \nabla \in \uPdelta$ (utiliser l'appartenance à la Cramer 
avec $G = \nabla$ et $F \in \uPdelta$).
D'après~\ref{MiniWiebe} (conséquence du théorème de Wiebe), on obtient $\omega(F) = 0$.
On peut fournir une autre preuve : 
comme $F$ est dans $\uPdelta$, 
le scalaire $\omega(F)$ est combinaison $\bfA$-linéaire de mineurs 
d'ordre $s_\delta + 1$ de $\Syl_\delta$, et ces 
derniers mineurs sont nuls, d'après~\ref{IdeauxDeterminantielsKP}.

iii) Cela résulte de i) car $\omega$ est linéaire. 
Pour le \og en particulier \fg{}, on utilise que 
$\nabla - \nabla' \in \uPdelta$ d'après~\ref{NablaDansLeSature}.

iv) D'après~\ref{NablaEtalon}, on sait que pour le jeu étalon, 
le mouton noir $X^\emouton$ est un déterminant bezoutien.
On conclut avec le point~{ii)} et le fait que l'endomorphisme $W_{1,\delta}$
est normalisé.

v) Résulte du point i) (appartenance à la Cramer).

vi) Soit $F \in \uPsat_\delta$. D'après le point v), on a 
$X^\alpha \omega(F) \in \uPsat$ pour tout $|\alpha | = \delta$.
Donc $\omega(F) \in \uPsat$.

Pour le \og en particulier \fg{}, 
rappelons que $\nabla$ appartient à $\uPsat_\delta$. 
Enfin, en générique, $\omega(\nabla)$, vu comme un polynôme en
les coefficients des $P_i$, est un polynôme 
primitif par valeur (puisque sa spécialisation en le jeu étalon vaut $1$) ;
c'est donc un élément régulier.
\end{proof}

\begin{rmq}
Apportons quelques précisions pour le point vi).

Pour le polynôme $F = \nabla \in \uPsat_\delta$, l'appartenance
$\omega(\nabla) \in \ElimIdeal$ peut être contrôlée de la manière
suivante.  Comme pour chaque $i$, on a $X_i \nabla \in \langle \uP
\rangle_{\delta + 1}$, en multipliant par $X_i$:
$$
\forall\, |\alpha | =\delta, \qquad 
X^\alpha \omega(\nabla) 
\, \in \, 
\uPdelta \, + \, \bfA \nabla
$$
on obtient $X^\gamma\, \omega(\nabla) \in \langle \uP \rangle_{\delta +1}$
pour tout $|\gamma|=\delta+1$.
Ainsi $\omega(\nabla) \in \Ann(\bfB_{\delta + 1})$, 
plus précis que l'appartenance  $\omega(\nabla) \in \ElimIdeal$.
Cependant lorsque $\uP$ est une suite régulière, voir~\ref{AnnEqualities}.
\end{rmq}

\subsection{Sigma version des objets}
\label{SigmaVersion}

Dans un certain nombre de constructions telles que $W_1$ (cf
\ref{DefW1}), $W_\calM$ (cf \ref{DefWM}) et $\Omega$ (cf
\ref{DefEndoOmega}), nous avons privilégié un ordre de parcours des
indéterminées à savoir l'ordre $(X_1, \dots, X_n)$.  Nous allons
devoir modifier cet ordre de parcours en permutant les $X_i$ à l'aide
d'une permutation $\sigma \in \fS_n$. Mais il est important de signaler
que l'on ne change surtout pas le format $D = (d_1, \dots, d_n)$ de
degrés. Dans la définition suivante, ce n'est pas lui qui subit une
permutation par $\sigma$!

\index{sigma@$\sigma$-version des objets ($\sigma\in\fS_n$)}%

\begin{defn}\label{DefWMsigma}

Soit $\calM$ un sous-module monomial de $\Jex_1$. Pour $\sigma \in
\fS_n$, 
on note $W^\sigma_M(\uP)$ l'endomorphisme du $\bfA$-module $\calM$
défini sur sa base monomiale par :
$$
W_\calM^\sigma :\ X^\alpha \ \longmapsto \ 
\pi_\calM \bigg(\dfrac{X^\alpha}{X_i^{d_i}} P_i \bigg)
\qquad
\begin {array}{l}
\hbox {où $X_i$ est le premier pour l'ordre de parcours} \\
\hbox{$X_{\sigma(1)}, \dots, X_{\sigma(n)}$ tel que $X_i^{d_i}$ divise $X^\alpha$}
\\
\end {array}
$$
Cet endomorphisme est normalisé, au sens où pour le jeu étalon, on a 
$W^\sigma_\calM(\uX^D) = \Id$.
\end{defn}

\label{NOTA05-WMsigma}%
%
%
\index{W@les endomorphismes!z@$W^\sigma_\calM(\uP)$ ($\sigma\in\fS_n$, $\calM\subset\Jex_1$)}%

\subsubsection*{Tordre la relation d'ordre classique sur $\{1..n\}$ par $\sigma \in \fS_n$} 

\index{ordre sur $\{1..n\}$ tordu par une permutation}%

De manière équivalente, il est préférable  de considérer l'ordre $<_\sigma$ 
sur $\{1..n\}$ que nous avons déjà introduit et utilisé en \ref{NOTA04-ordreTordu}
(page \pageref{NOTA04-ordreTordu}) et \ref{NOTA04-omegasigma} (page \pageref{NOTA04-omegasigma}):
$$
\sigma(1) \ <_\sigma \ \cdots \ <_\sigma\ \sigma(n)
$$
On a donc 
$$
\sigma(i') <_{\sigma} \sigma(j') 
\ \Leftrightarrow \ 
i' < j' 
\qquad \text{c'est-à-dire } \qquad
i <_{\sigma} j
\ \Leftrightarrow \ 
\sigma^{-1}(i) < \sigma^{-1}(j) 
$$
Associé à cet ordre, 
on définit $\smin(E)$ comme étant le plus petit élément de $E$ au sens $<_\sigma$.
Précisément, pour une partie $E \subset \{1..n\}$, 
$$
\smin (E) \ = \ 
\left\{
\begin{array}{ll}
\sigma \big(\min\{i'\in \{1..n\}\text{\ tel que\ } \sigma(i') \in E\}\big)
    & \text{si $E \neq \emptyset$} \\
n+1 & \text{sinon} \\
\end{array}
\right.
$$

\begin{rmq}
Soit $\sigma \in \fS_n$ l'involution renversante de $\{1..n\}$ i.e.
$\sigma = (1,n)(2,n-1) \cdots$, qui réalise $i \mapsto n-i+1$.
Il est clair que $\smin(E) = \max(E)$ et que l'ordre $<_\sigma$ est exactement
l'ordre~$>$ usuel.
\end {rmq}

Pour $E = \DivSeq(X^\alpha)$, on note $\sminDiv(X^\alpha)$ pour $\smin(E)$.
L'endomorphisme $W^\sigma_\calM$ défini ci-dessus s'écrit alors
$$
W^\sigma_\calM :\ 
X^\alpha \ \longmapsto \ 
\pi_{\calM} 
\bigg(\dfrac{X^\alpha}{X_i^{d_i}} P_i \bigg)
\qquad  \text{où $i = \sminDiv(X^\alpha)$}
$$
\label{NOTA05-sminDiv}%
\index{mécanisme de sélection!$\sminDiv$}%
De manière analogue, la forme linéaire 
$\omega$ introduite en~\ref{DefEndoOmega} possède une version tordue par 
$\sigma \in \fS_n$ explicitée ci-dessous.

\begin{defn}
Pour $F \in \bfA[\uX]_\delta$, l'endomorphisme $\Omega^\sigma(F)$ est défini par :
$$
\Omega^\sigma(F) : \ \bfA[\uX]_\delta\ \longrightarrow \ \bfA[\uX]_\delta, 
\qquad \qquad
X^\alpha \longmapsto
\left\{
\begin{array}{ll}
\dfrac{X^\alpha}{X_i^{d_i}} P_i & 
\text{si $X^\alpha \in \Jex_{1,\delta}$ où $i = \sminDiv(X^\alpha)$} \\ [0.6cm]
F & 
\text{si $X^\alpha = X^\emouton$} 
\end{array}
\right.
$$
La forme linéaire $\omega^\sigma$ est  
$$
\omega^\sigma : \bfA[\uX]_\delta \longrightarrow \bfA, 
\qquad 
F \longmapsto \det \big( \Omega^\sigma(F) \big)
$$
et on a 
$$
\omega^\sigma(X^{\emouton}) 
\ = \ 
\det W_{1,\delta}^\sigma
$$
\end{defn}

\label{NOTA05-Omegasigma}%
\label{NOTA05-omegasigma}%
\index{forme linéaire!$\omega^\sigma=\omega^\sigma_\uP$ (pilotée par $\sminDiv$)}%
\index{O@Omega-constructeur à valeurs dans!$\End_\bfA(\bfA[\uX]_\delta)$!2@$\Omega^\sigma=\Omega^\sigma_\uP$}%

\subsubsection*{$\bullet$ Lorsque la $\sigma$-version ne fait aucun effet}

Il peut arriver que deux permutations distinctes $\sigma, \sigma' \in \fS_n$ conduisent aux
mêmes objets tordus. En voici un cas extrême dû au fait que les ensembles
de divisibilité $\DivSeq(X^\alpha)$ qui interviennent sont de cardinal~1.

\begin{prop} [Cas exceptionnels]\label{J2dNul} \leavevmode

\begin{enumerate}[\rm i)]
\item 
Si $\Jex_{2,\delta} = 0$, alors $\omega^\sigma= \omega$ pour tout $\sigma$.

\item 
Si $\Jex_{2,d} = 0$, alors $\det W_{1,d}^\sigma = \det W_{1,d}$ pour tout $\sigma$.
\end{enumerate}
\end{prop}

\subsubsection*{$\bullet$ Un exemple détaillé où pour $n=3$, il y a $6$ permutations $\sigma \in \fS_3$
                 mais seulement 4 $\sigma$-objets}

Prenons celui de $D = (d_1,\, d_2=1,\, d_3)$ 
en degré $\delta = (d_1-1) + (d_3-1)$.  Puisque $d_1 + d_3 > \delta$, les parties de $\{1,2,3\}$ contenant $\{1,3\}$
ne sont pas des ensembles de $D$-divisibilité en ce
degré. Comme ensemble $E$ de divisibilité, hormis les singletons, il
reste donc comme possibilité $E = \{1,2\}$ et/ou $E = \{2,3\}$. Cela partitionne les 6
permutations de $\fS_3$ en 4 classes
$$
\{\Id_3\},\qquad \{(1,3)\}, \qquad \{(1,2,3), (1,2)\}, \qquad \{(1,3,2),(2,3)\}
$$
au sens où deux permutations de la même classe ont même $\smin(E)$ quel que soit
$E \subset \{1,2,3\}$ ensemble de $D$-divisibilité en degré $\delta$ et fournissent
donc le même objet tordu.
Pour être complètement exhaustif, dressons le tableau 
indiquant les $\sminDiv$ de chaque ensemble de $D$-divisibilité 
selon la permutation~$\sigma$:
$$
\begin{array}{c|c|c|c|c|c|}  
\sigma \backslash E & \{1\} & \{2\} & \{3\} & \{1,2\} & \{1,3\} \\
\hline 
\id & 1 & 2 & 3 & 1 & 1 \\
\hline 
(1,3) & 1 & 2 & 3 & 2 & 3 \\
\hline 
(1,2) & 1 & 2 & 3 & 2 & 1 \\
\hline 
(2,3) & 1 & 2 & 3 & 1 & 1 \\
\end{array}
$$
Particularisons cet exemple à $D = (3,1,2)$ dont voici les ensembles
de divisibilité en degré $\delta = 3$ :
$$
\begin {array}{|c|c|c|c|c|c|c|c|c|c|c|}  
X^3 &X^2Y & \overset{\!\mouton}{X^2Z} &XY^2 &XYZ &XZ^2 &Y^3 &Y^2Z &YZ^2 &Z^3 
\\[1mm]
\hline
\vrule height12pt depth3pt width0pt
1 & 2 &\emptyset & 2 & 2 & 3 & 2 & 2 & 2,3 & 3 
\\[1mm]
\end{array}
$$
Ici, il n'y a plus que 2 classes au lieu de 4: celle constituée des
$\sigma$ telles que $\sigma(2) < \sigma(3)$ et la classe
complémentaire $\sigma(2) > \sigma(3)$. Utiliser un ordre tordu n'a
rien à voir avec le fait de permuter la base comme 
ci-dessous où $\Omega$ est donnée dans deux bases permutées
l'une de l'autre.
$$
\EastBordermatrix{
a_{1}& b_{1}& \sbullet & .& .& c_{1}& .& .& .& . & X^{3}\\ 
a_{2}& b_{2}& \sbullet & b_{1}& .& c_{2}& .& .& .& . & X^{2}Y\\ 
a_{3}& b_{3}& \sbullet& .& b_{1}& c_{3}& .& .& .& c_{1} & X^{2}Z \\ 
a_{4}& .& \sbullet & b_{2}& .& c_{4}& b_{1}& .& .& . & XY^{2} \\ 
a_{5}& .& \sbullet & b_{3}& b_{2}& c_{5}& .& b_{1}& .& c_{2} & XYZ \\ 
a_{6}& .& \sbullet & .& b_{3}& c_{6}& .& .& b_{1}& c_{3} & XZ^{2} \\ 
a_{7}& .& \sbullet & .& .& .& b_{2}& .& .& . & Y^{3} \\ 
a_{8}& .& \sbullet & .& .& .& b_{3}& b_{2}& .& c_{4} & Y^{2}Z\\ 
a_{9}& .& \sbullet & .& .& .& .& b_{3}& b_{2}& c_{5} & YZ^{2}\\ 
a_{10}& .& \sbullet & .& .& .& .& .& b_{3}& c_{6} & Z^{3}\\ 
}
\quad 
\EastBordermatrix{
c_{6}& b_{3}& .& .& .& .& .& \sbullet & .& a_{10} & Z^{3}\\ 
c_{5}& b_{2}& b_{3}& .& .& .& .& \sbullet & .& a_{9} & YZ^{2}\\ 
c_{4}& .& b_{2}& b_{3}& .& .& .& \sbullet & .& a_{8} & Y^{2}Z \\ 
.& .& .& b_{2}& .& .& .& \sbullet & .& a_{7} & Y^{3} \\ 
c_{3}& b_{1}& .& .& c_{6}& b_{3}& .& \sbullet & .& a_{6} & XZ^{2}\\ 
c_{2}& .& b_{1}& .& c_{5}& b_{2}& b_{3}& \sbullet & .& a_{5} & XYZ \\ 
.& .& .& b_{1}& c_{4}& .& b_{2}& \sbullet & .& a_{4} & XY^{2}\\ 
c_{1}& .& .& .& c_{3}& b_{1}& .& \sbullet & b_{3}& a_{3} & X^{2}Z\\ 
.& .& .& .& c_{2}& .& b_{1}& \sbullet & b_{2}& a_{2} & X^{2}Y\\ 
.& .& .& .& c_{1}& .& .& \sbullet & b_{1}& a_{1} & X^{3}\\ 
}
$$
Par contre voici l'effet de l'ordre $<_\sigma$ pour $\sigma = (1,3)$
qui appartient à la classe $\sigma(2) > \sigma(3)$.
A droite, on a fait figurer le poids en $P_i$ (nombre de colonnes de type $P_i$) pour chacune des $\Omega^\sigma$.
$$
\Omega^{(1,3)} =
\EastBordermatrix{
a_{1}& b_{1}& \sbullet & .& .& c_{1}& .& .& .& . & X^{3} \\ 
a_{2}& b_{2}& \sbullet & b_{1}& .& c_{2}& .& .& c_{1}& . & X^{2}Y \\ 
a_{3}& b_{3}& \sbullet & .& b_{1}& c_{3}& .& .& .& c_{1} & X^{2}Z \\ 
a_{4}& .& \sbullet & b_{2}& .& c_{4}& b_{1}& .& c_{2}& . & XY^{2} \\ 
a_{5}& .& \sbullet & b_{3}& b_{2}& c_{5}& .& b_{1}& c_{3}& c_{2} & XYZ \\ 
a_{6}& .& \sbullet & .& b_{3}& c_{6}& .& .& .& c_{3} & XZ^{2} \\ 
a_{7}& .& \sbullet & .& .& .& b_{2}& .& c_{4}& . & Y^{3} \\ 
a_{8}& .& \sbullet & .& .& .& b_{3}& b_{2}& c_{5}& c_{4} & Y^{2}Z \\ 
a_{9}& .& \sbullet & .& .& .& .& b_{3}& c_{6}& c_{5} & YZ^{2} \\ 
a_{10}& .& \sbullet & .& .& .& .& .& .& c_{6} & Z^{3} \\ 
}
\qquad\qquad\qquad
\begin {array}{c|c|c|c}  
\sigma  &P_1  &P_2 &P_3 \\[0.5mm]
\hline
\vrule height12pt depth3pt width0pt
\Id_3 &1        &6       &2  \\[0.5mm]
\hline
\vrule height12pt depth3pt width0pt
(1,3) &1        &5       &3   \\
\end {array}
$$

\subsection{Décomposition de $\Jex_1$ et divisibilité $\det W_{1,d}(\protect\uP)\mid\det W_{1,d+1}(\protect\uP)$
            pour $d\geqslant\delta +1$}
\label {SousSectionDecompositionJ1circJ1plus}

L'égalité du premier point du théorème suivant figure chez
Demazure \cite[\S6]{Demazure1} avec une preuve semblable à la nôtre.
Elle est également abordée par Jouanolou en \cite[3.6.9.1 p. 138]{J7}
(la numérotation y est incorrecte, elle devrait être 3.9.6.1).
Jouanolou utilise une structure triangulaire en y faisant intervenir
la suite $\uP' = (P_1,\dots, P_{n-1}, X_n^{d_n})$ 
ce qui rend son argumentation un tantinet plus complexe que la nôtre.
Il est amené à utiliser la formule de Macaulay pour $\uP$ et $\uP'$ en degré $d \geqslant \delta+1$ 
(ce qui suppose que le résultant soit en place et que cette formule soit acquise). 
Il mentionne, dans la remarque 3.9.6.8, p.139, qu'il n'utilise 
pas toute la force de la formule de Macaulay mais
seulement le fait que $\Res(\uP) \mid \det W_{1,d}(\uP)$ 
avec un quotient ne dépendant pas de $P_n$.

\begin{theo} \label{DetW1dDiviseDetW1dplus1}
Soit $d \geqslant \delta+1$.

\begin {enumerate} [\rm i)]
\item
Soit $\uP' = (P'_1, \ldots, P'_{n-1})$ défini par $P'_i = P_i(X_n := 0) \in \bfA[X_1, \ldots, X_{n-1}]$. Alors
$$
\det W_{1,d+1}(\uP) \ =\  
\det W_{1,d}(\uP) \, \det W_{1,d+1}(\uP') 
$$
\item
Pour tout $\sigma \in \fS_n$, on a $\det W_{1,d}^\sigma(\uP) \mid \det W_{1,d+1}^\sigma(\uP)$ avec
un cofacteur ne dépendant pas de $P_{\sigma(n)}$.  
\end {enumerate}  
\end{theo}

Nous allons en fait obtenir un résultat plus général et plus précis, qui se situe
au niveau des endomorphismes et pas seulement des déterminants. A cet effet,
indépendamment du format $D$, nous allons partager les monômes en deux catégories
selon leur divisibilité par $X_n$.

\begin{defns} \leavevmode

\begin {enumerate} [\rm i)]
\item
Soit $\calM$ un sous-module monomial de $\bfA[\uX]$. On note $\calM^+$
le sous-module de $\calM$ de base les $X^\alpha \in \calM$ tels
que $\alpha_n \geqslant 1$, $\calM^\circ$ son supplémentaire monomial
(de base les $X^\alpha$ tels que $\alpha_n = 0$) et $\calM^+/X_n$ le sous-module
monomial de base les $X^\alpha/X_n$ avec $X^\alpha \in \calM^+$.
En particulier $\calM = \calM^+ \oplus \calM^\circ$.

\item
On note $\uX' = (X_1, \dots, X_{n-1})$ et on désigne par $\calM'$ le
sous-module monomial de $\bfA[\uX'] \subset \bfA[\uX]$ de base les
$X^\alpha \in \calM^\circ$. Ainsi $\calM'$ n'est autre que $\calM^\circ$ mais
vu dans $\bfA[\uX']$.

\item
Etant donné un système $\uP = (P_1,\dots,P_n)$ de $\bfA[\uX]$, on lui associe,
en posant $P'_i = P_i(X_n := 0)$, deux systèmes, le premier dans $\bfA[\uX]$,
le second dans $\bfA[\uX']$:  
$$
\uP^\circ = (P'_1, \dots, P'_n),  \qquad \uP' = (P'_1, \dots, P'_{n-1})
$$
\end {enumerate}
\end{defns}

On rappelle que pour $d' \geqslant \delta+1$, on a $\Jex_{1,d'} = \bfA[\uX]_{d'}$.
Cependant, dans le lemme et la proposition ci-dessous, nous avons choisi d'utiliser
$\Jex_{1,\sbullet}$ au lieu de $\bfA[\uX]_{\sbullet}$ pour attirer l'attention
sur le fait que $W_\calN(\uP)$ n'est défini que si $\calN \subset \Jex_1$.

\begin{lem}
Soit $d \geqslant \delta+1$ et $\calM \subset \Jex_{1,d+1}$.

\begin{enumerate}[\rm i)]
\item
On a ${\calM^+/X_n} \subset \Jex_{1,d}$.
De plus, pour $\calM = \Jex_{1,d+1}$, on a d'une part $\calM^+/X_n = \Jex_{1,d}$
et d'autre part, dans $\bfA[\uX']$, $\calM' = \Jex_{1,d+1}$.

\item
L'application linéaire monomiale suivante est un isomorphisme qui conserve $\minDiv$:
$$
\Psi : \ 
\begin{array}[t]{rcl}
\calM^+ & \buildrel {\simeq} \over \longrightarrow & \calM^+/X_n \\ [0.1cm]
X^{\alpha} & \longmapsto & \dfrac{X^{\alpha}}{X_n} 
\end{array}
$$
\end{enumerate}
\end {lem}

\begin {proof} \leavevmode

\noindent  
i) On a l'inclusion $\calM^+/X_n \subset \bfA[\uX]_d$ équivalente, puisque $d \ge
\delta+1$, à l'inclusion $\calM^+/X_n \subset \Jex_{1,d}$.
En ce qui concerne l'inclusion $\Jex_{1,d} \subset \Jex_{1,d+1}^+/X_n$:
soit $X^\alpha \in \Jex_{1,d}$; alors $X_n X^\alpha \in \Jex_{1,d+1}^+$,
donc $X^\alpha = (X_n X^\alpha)/X_n \in \Jex_{1,d+1}^+/X_n$.
L'égalité concernant $\calM'$ dans $\bfA[\uX']$ est immédiate.

\smallskip
\noindent
ii)
Soit $X^\alpha \in \calM^+$ et $i := \minDiv(X^\alpha)$.  On sait déjà que $X^\alpha/X_n \in \Jex_1$
i.e. que $\DivSeq(X^\alpha/X_n)$ est non vide.

\smallskip
\noindent
Si $i < n$, il est clair que $i = \minDiv(X^\alpha/X_n)$.
Et dans le cas $i = n$, on a $\DivSeq(X^\alpha) = \{n\}$ donc 
aussi $\DivSeq(X^\alpha/X_n) =\{n\}$ d'où l'égalité des $\minDiv$.
\end {proof}

\begin{prop} \label{XnDecompositionJ1d}
Soit $d \geqslant \delta+1$ et $\calM \subset \Jex_{1,d+1}$.

\begin{enumerate}[\rm i)]
\item 
L'endomorphisme $W_\calM(\uP)$ stabilise $\calM^+$ :
$$
W_\calM \ = \ 
\EastBordermatrix{
\,W_{\calM^+}    & \star   & \calM^+ \\ \noalign {\medskip}
0           & W_{\calM^\circ} & \calM^\circ
}
$$

\item
L'isomorphisme $\Psi$ conjugue les endomorphismes $W_{\calM^+}(\uP)$ et $W_{\calM^+/X_n}(\uP)$.

\item
On a les égalités $W_{\calM^\circ}(\uP) =  W_{\calM^\circ}(\uP^\circ) = W_{\calM'}(\uP')$.
En conséquence:
$$
\det W_\calM(\uP) \ = \ 
\det W_{\calM^+/X_n}(\uP) \times 
\det W_{\calM'}(\uP') 
$$
\end{enumerate}
\end{prop}

\begin{proof} \leavevmode

\smallskip
\noindent
i) Soit $X^\alpha \in \calM^+$ et $i=\minDiv(X^\alpha)$. L'image de $X^\alpha$ par 
$W_\calM$ est une combinaison linéaire de monômes de $\calM$ 
du type $X^\beta = (X^\alpha/X_i^{d_i}) X^\gamma$ où $X^\gamma$ est un monôme de $P_i$
et il faut montrer que $X^\beta \in \calM^+$.

\smallskip
\noindent
Supposons $i < n$. Alors $\beta_n = \alpha_n + \gamma_n$ et puisque $\alpha_n \geqslant 1$, on a 
$X^\beta \in \calM^+$.

\noindent
Supposons $i = n$. Comme $\Psi$ conserve $\minDiv$, on a
$\minDiv(X^\alpha/X_n) = n$, a fortiori $X_n^{d_n} \mid X^\alpha/X_n$
donc $X_n \mid X^\alpha/X_n^{d_n}$; on en déduit $X_n \mid X^\beta$,
c'est-à-dire $X^\beta \in \calM^+$.

\medskip
\noindent
ii) Soit $X^\alpha \in \calM^+$ et $i = \minDiv(X^\alpha)$ égal, d'après le lemme précédent,
à $\minDiv(X^\alpha/X_n)$. Il faut vérifier que le diagramme suivant est commutatif
$$
\xymatrix @R = 1.0cm @C = 2.5cm @M=0.4pc{
X^\alpha \ar[d]_-{W_{\calM^+}(\uP)} \ar[r]^{\Psi} &
     X^\alpha/X_n \ar[d]^-{W_{\calM^+/X_n}(\uP)}
\\
\pi_{M^+}\big((X^\alpha/X_i^{d_i})P_i\big) \ar[r]^{\Psi} & 
     \pi_{\calM^+/X_n}\big((X^\alpha/X_i^{d_i}X_n)P_i\big)
\\
}
$$
D'après le point précédent appliqué à $\calM = \bfA[\uX]_d$, tout monôme $X^\beta$
de $(X^\alpha/X_i^{d_i})P_i$ est divisible par~$X_n$ et il suffit donc de vérifier
que
$$
\pi_{\calM^+} (X^\beta) = \pi_{\calM^+/X_n} (X^\beta/X_n) 
$$
ce qui est immédiat.

\smallskip
\noindent
iii) Sur les monômes $X^\alpha$ de $\calM^\circ = \calM'$, la sélection
$\minDiv$ ne prend pas la valeur $n$ puisque $\alpha_n = 0$.
Ceci fait que ni les monômes des $P_i$ divisibles par $X_n$, ni $P_n$
n'interviennent dans $W_{\calM^\circ}(\uP)$ et conduit à l'égalité
des endomorphismes.

\noindent
D'après i), on a:
$$
\det W_\calM(\uP) \ = \ 
\det W_{\calM^+}(\uP) \times 
\det W_{\calM^\circ}(\uP) 
$$
Mais d'après le point ii), on a $\det W_{\calM^+}(\uP) = \det W_{\calM^+/X_n}(\uP)$.
Et d'autre part
$$
\det W_{\calM^\circ}(\uP) =  \det W_{\calM^\circ}(\uP^\circ) = \det W_{\calM'}(\uP')
$$
d'où le résultat.
\end{proof}

\begin{proof}[Preuve du théorème~\ref{DetW1dDiviseDetW1dplus1}] \leavevmode

\noindent
i)  On applique le point iii) de la proposition précédente à $\calM = \Jex_{1,d+1}$.
D'après le lemme, on a d'une part $\calM^+/X_n = \Jex_{1,d}$ et d'autre part
dans $\bfA[\uX']$, $\calM' = \Jex_{1,d+1}$, ce qui donne le résultat. 

\medskip
\noindent
ii) D'après le premier point, la propriété est acquise pour $\sigma = \id$.
Cette propriété s'appuie sur le  mécanisme de sélection $\minDiv$ que
l'on peut remplacer par $\sminDiv$, d'où le résultat.

\smallskip
\noindent
Cependant, si le lecteur a besoin d'expliciter le cofacteur, voici quelques détails
qui pourront lui être utiles.
Il faut bien entendu adapter à $\sigma$ les définitions de $\calM^+,\calM^\circ,\calM'$ et
faire intervenir $\calM^+/X_{\sigma(n)}$. La suite $\uX'$ désigne maintenant
$(X_1, \cdots, X_n)$ privée de $X_{\sigma(n)}$ et $\uP'$ le système de $\bfA[\uX']$ défini par:
$$
P'_i = P_i(X_{\sigma(n)} := 0) \qquad\qquad
\uP' = (P'_1, \cdots, P'_n) \setminus \{P'_{\sigma(n)}\}
$$
Enfin, soit $\sigma' \in \fS_{n-1}$ définie pour $1 \leqslant i \leqslant n-1$ par
la formule de gauche. Avec un exemple à droite:
$$
\sigma'(i) = \begin {cases}
  \sigma(i)   &\hbox{si } \sigma(i) < \sigma(n) \\
  \sigma(i)-1 &\hbox{si } \sigma(i) > \sigma(n) \\
\end {cases}
\qquad\qquad\quad
\begin{array}{c|c|c|c|c|c|c|c|c|c|}  
       &1 &2 &3 &4 &5 &6 &7 &8 &9 \\
\hline 
\sigma &7 &3 &1 &6 &9 &8 &2 &4 &5 \\
\hline 
\sigma'&6 &3 &1 &5 &8 &7 &2 &4\\
\end{array}
$$
Cette permutation $\sigma' \in \fS_{n-1}$ est caractérisée par la propriété suivante: la
restriction de $\sigma$ à $\{1..n-1\}$ induit un isomorphisme
d'ensembles ordonnés $\sigma : \big(\{1..n-1\}, <_{\sigma'}\big)
\buildrel \simeq \over\longrightarrow \big(\sigma(\{1..n-1\}),
<_\sigma\big)$.

\noindent
Alors:
$$
\det W^\sigma_\calM(\uP) \ = \ \det W^\sigma_{\calM^+/X_{\sigma(n)}}(\uP) \times  \det W^\sigma_{\calM^\circ}(\uP)
\ = \ \det W^\sigma_{\calM^+/X_{\sigma(n)}}(\uP) \times  \det W^{\sigma'}_{\calM'}(\uP') 
$$
En particulier:
$$
\det W^\sigma_{1,d+1}(\uP) \ =\  \det W^\sigma_{1,d}(\uP) \, \det W^{\sigma'}_{1,d+1}(\uP')
$$
\end{proof}

\subsection{Une suite $\protect \uP$ régulière un peu spéciale}

Considérons $\uP = (Y,X,Z^3)$ de format $D = (1,1,3)$ de
$\delta=2$. Il s'agit d'une suite régulière, mais qui présente un
certain caractère récalcitrant: la plupart des objets matriciels que
nous avons attachés à~$\uP$ dans les sections précédentes ne
présentent aucune utilité: ils sont triviaux.  Afin de mieux contrôler
les calculs, nous allons utiliser une forme avoisinante, qui est la
suite généralisée $(aY, bX, cZ^3)$, où $a,b,c$ sont des indéterminées.

\index{système!z@$(aY,bX,cZ^3)$}%

\begin{prop} \label{YXZ3}
Soit $\uP = (aY, bX, cZ^3)$ de $\delta = 2$, $d \geqslant \delta$ et $\sigma \in \fS_3$ quelconque. 

\begin{enumerate}[\rm i)]
\item La forme linéaire $\omega^\sigma$ est nulle. 
En particulier, $\det W_{1,\delta}^\sigma = 0$.

\item  On a $\det W_{1,d}^\sigma = 0$.

\item On a $\det W_{2,d}^\sigma = 0$. Et pour $d = \delta$, $W_{2,d}^\sigma$
est la matrice nulle $1 \times 1$.
\end{enumerate}
\end{prop}

\begin{proof} 
i) Pour $F \in \bfA[X,Y,Z]_2$, l'endomorphisme $\Omega^\sigma(F)$
transforme les deux vecteurs de base distincts ci-dessous $X^2,Y^2$
en deux colonnes proportionnelles
$$
X^2 \ \mapsto \ \dfrac{X^2}{X} P_1 = aXY 
\qquad \text{et } \qquad 
Y^2 \ \mapsto \ \dfrac{Y^2}{Y} P_2 = bXY 
$$
Ainsi son déterminant $\omega^\sigma(F)$ est nul. 
On a donc montré que $\omega^\sigma = 0$.

\medskip
\noindent
Soit $d \geqslant \delta = 2$ et $\sigma \in \fS_3$.

ii) Supposons que pour l'ordre $<_{\sigma}$, l'entier $3$ soit le plus grand. Alors
$$
W_{1,d}^\sigma : \qquad  X^2 Z^{d-2} \ \mapsto \ aXY Z^{d-2},  \qquad
Y^2 Z^{d-2} \ \mapsto \ bXY Z^{d-2} 
$$
Ainsi, on a $\det W_{1,d}^\sigma = 0$. Idem si $1$ est le plus grand
pour l'ordre $<_{\sigma}$:
$$
W_{1,d}^\sigma : \qquad X^2 X^{d-2} \ \mapsto \ aXY X^{d-2}, \qquad
Y^2 X^{d-2} \ \mapsto \ bXY X^{d-2} 
$$
On laisse le soin au lecteur de traiter le cas qui reste ($2$ maximum pour l'ordre $<_{\sigma}$).

\medskip

iii)
Les monômes $XY^{d-1}$ et $X^{d-1} Y$ sont dans $\Jex_{2,d}$. 
Suivant la valeur de $\sigma$, l'image de l'un de ses monômes par $W_{2,d}^\sigma$ est nulle.

\noindent
Si $1 <_\sigma 2$, on a 
$W_{2,d}^\sigma(XY^{d-1}) = \pi_{\Jex_2}\Big(\dfrac{XY^{d-1}}{X} aY \Big) = \pi_{\Jex_2}(aY^{d}) = 0$.

\noindent
Si $2 <_\sigma 1$, on a 
$W_{2,d}^\sigma(X^{d-1}Y) = \pi_{\Jex_2}\Big(\dfrac{X^{d-1}Y}{Y} bX \Big) = \pi_{\Jex_2}(bX^{d}) = 0$.

\medskip
Pour $d = \delta$, on a $\Jex_{2,\delta} = \bfA XY$ de dimension 1 donc $W^\sigma_{2,\delta} = 0$.
\end{proof}

\begin{rmq}
Soit $\uP = (P_1, P_2, P_3)$ de même format $D = (1,1,3)$. 
On laisse le soin au lecteur de vérifier que $\det W_{2,\delta}^\sigma(\uP)$ est
ou bien $\coeff_X(P_1)$ ou bien $\coeff_Y(P_2)$.  On comprendra plus
tard que c'est essentiellement la nullité de $\det W_{2,\delta}^\sigma$ qui est en cause.
En effet, celle de $\det
W_{2,\delta}$ entraîne celle de~$\omega$ puisque chaque coefficient de
$\omega$ est multiple de $\det W_{2,\delta}$ (cf. section~\ref{sousSectionMacaulayRecurrence}).
Et également celle de $\det W_{2,d}$ pour $d \geqslant \delta$ puisque
$\det W_{2,d}$ est multiple de $\det W_{2,d'}$
pour $d \geqslant d'$ (cf.~\ref{DivisibiliteEntreW2}).  Idem
pour $\sigma$ quelconque.
\end{rmq}

\subsection{Homogénéité et poids en $P_i$ des divers déterminants}

\index{poids (en $P_i$)}%

Ici, nous supposons le système $\uP$ générique, ainsi $\bfA = \bfk[\indetsPi]$.  Il
s'agit de prouver que les divers déterminants introduits précédemment,
vus dans $\bfA$ comme des polynômes en les coefficients des $P_i$,
sont homogènes en chaque $P_i$ et de préciser ce poids d'homogénéité.

Nous allons voir que ce poids d'homogénéité n'est autre qu'une dimension
d'un sous-module monomial ad-hoc de $\bfA[\uX]$, dépendant du déterminant
considéré. A la base, comme on le voit dans le premier point de la
proposition suivante, il y a le déterminant $\det W_\calM(\uP)$ où
$\calM$ est un sous-module monomial de~$\Jex_{1,d}$.
La généralisation aux endomorphismes 
$W_{\calM}^\sigma$ est immédiate: il suffit de tordre la relation d'ordre 
classique sur $\{1..n\}$ par $\sigma \in \fS_n$ 
comme indiqué en~\ref{SigmaVersion} ; 
concrètement, il s'agit de remplacer toutes les occurrences de 
$\minDiv$ par $\sminDiv$ !

\medskip

Commençons par donner un exemple: celui de la matrice de $W_{\calM}$
pour $\calM = \Jex_{2,d}$ avec $D = (3,1,2)$ et $d=5$ :
$$
W_{2,5} \ = \ 
\EastBordermatrix{
a_{1} & . & . & . & . & . & . & . & . & . & \Heti{X^{4}Y} \\ 
a_{2} & a_{1} & . & . & . & . & . & . & . & . & \Heti{X^{3}Y^{2}} \\ 
a_{3} & . & a_{1} & . & . & . & . & . & . & . & \Heti{X^{3}YZ} \\ 
. & . & . & a_{1} & b_{1} & . & . & . & . & . & \Heti{X^{3}Z^{2}} \\ 
a_{6} & . & a_{3} & a_{2} & b_{2} & b_{1} & . & . & . & . & \Heti{X^{2}YZ^{2}} \\ 
a_{9} & a_{6} & a_{5} & a_{4} & . & b_{2} & . & b_{1} & . & . & \Heti{XY^{2}Z^{2}} \\ 
a_{10} & . & a_{6} & a_{5} & . & b_{3} & b_{2} & . & b_{1} & . & \Heti{X YZ^{3}} \\ 
. & a_{9} & a_{8} & a_{7} & . & . & . & b_{2} & . & . & \Heti{Y^{3}Z^{2}} \\ 
. & a_{10} & a_{9} & a_{8} & . & . & . & b_{3} & b_{2} & . & \Heti{Y^{2}Z^{3}} \\ 
. & . & a_{10} & a_{9} & . & . & . & . & b_{3} & b_{2} & \Heti{YZ^{4}} \\ 
}
$$
Il ne fait aucun doute que son déterminant est homogène en chaque $P_i$ de
poids le nombre de colonnes de type $P_i$ \idest{} de poids~4 en $P_1$,
de poids~6 en $P_2$ et de poids~0 en $P_3$. La proposition qui suit
fournit une formalisation de ce genre de résultat.

\begin{prop}[Poids]\leavevmode 
\label{PoidsDetW}

\begin{enumerate}[\rm i)]
\item 
Soit $\calM$ un sous-module monomial de $\Jex_{1,d}$.
Le déterminant $\det W_\calM$ est un polynôme homogène de~$\bfA$, de poids en $P_i$ égal à 
$\dim \calM^{(i)}$, où $\calM^{(i)}$ désigne le sous-module de $\calM$ 
de base les $X^\alpha$ tels que $\minDiv(X^\alpha) = i$

\item 
Soit $h \geqslant 1$.
Le poids en $P_i$ de $\det W_{h,d}$ est égal à $\dim \Jex_{h,d}^{(i)}$.
En particulier, $\det W_{h,d}$ ne dépend pas des $h-1$ derniers polynômes.
Plus généralement, pour $\sigma\in \fS_n$, $\det W_{h,d}^\sigma$ ne dépend pas de $P_i$ où $i$ parcourt les $h-1$ 
derniers éléments de $\big(\sigma(1), \dots, \sigma(n)\big)$.

\item 
Le déterminant $\det W_{1,d}$ est homogène en $P_n$ de poids 
$\dim \Jex_{1\setminus 2,d}^{(n)}$.

Plus particulièrement, pour $d \geqslant \delta$, ce poids vaut 
$$
\left \{
\begin{array}{ll}
\widehat d_n -1 & \text{si $d = \delta$} \\ [5pt]
\widehat d_n & \text{si $d \geqslant \delta+1$} \\ 
\end{array} 
\right.
\qquad \hbox {où\quad $\widehat d_n  = d_1 \cdots d_{n-1}$}
$$
Plus généralement, le poids de $\det W_{1,d}^\sigma$ en $P_{\sigma(n)}$ est 
$\dim \Jex_{1\setminus 2,d}^{(\sigma(n))}$, avec les mêmes précisions que ci-dessus
(en remplaçant $d_n$ par $d_{\sigma(n)}$).

\item 
Chaque scalaire $\omega(X^\beta)$ est homogène en $P_i$ 
de poids $\dim \Jex_{1,\delta}^{(i)}$, indépendant de $X^\beta$.

En particulier, il est homogène en $P_n$ de poids 
$\dim \Jex_{1\setminus 2, \delta}^{(n)} = \widehat d_n -1$. 

\medskip

Plus généralement, $\omega^\sigma(X^\beta)$ est homogène en $P_{\sigma(n)}$ 
de poids $\dim \Jex_{1\setminus 2, \delta}^{(\sigma(n))} = \widehat d_{\sigma(n)} -1$. 
\end{enumerate}
\end{prop}

\label{NOTA05-M^{(i)}}%
\label{NOTA05-Jexhdi}%
%
%
%

\begin{proof} \leavevmode
i)
Raisonnons matriciellement relativement à la base monomiale de $\calM$.
Les coefficients non nuls d'une colonne de
la matrice de $W_{\calM}$ sont les coefficients d'un polynôme
de type $P_i$ \idest{} de la forme  $(X^\alpha/X_i^{d_i})\,P_i$
dans la base fixée, où $i = \minDiv(X^\alpha)$.

\medskip
Ainsi, le poids en $P_i$ du déterminant est 
le nombre de colonnes de type $P_i$. Ce nombre est
égal au nombre de monômes $X^\alpha \in \calM$
tels que $\minDiv(X^\alpha) = i$.
Ce poids est donc $\dim \calM^{(i)}$.

\medskip
ii) C'est un cas particulier du point i) avec $\calM = \Jex_{h,d}$.

\medskip
iii)
Confer le lemme de dénombrement \ref{LemmeDenombrementJ1moins2id}.

\medskip
iv) 
Par définition, $\omega(X^\beta)$ est le déterminant 
de $\Omega(X^\beta)$. En utilisant un développement du déterminant 
par rapport à la colonne-argument, on voit que ce déterminant 
est, dans la matrice de Sylvester $\Syl_\delta$,
le mineur d'une sous-matrice obtenue en supprimant 
la ligne $X^\beta$, cf. le dessin à la page \pageref{omegaDessin}.
Chaque colonne de cette sous-matrice est constituée de coefficients
d'un polynôme de la forme $(X^\alpha/X_i^{d_i})\,P_i$ avec $i = \minDiv(X^\alpha)$,
colonne qualifiée de type $P_i$.
Le nombre de colonnes de type $P_i$ est le nombre de $X^\alpha \in \Jex_{1,\delta}$
tels que $i = \minDiv(X^\alpha)$ \idest{} la dimension annoncée.

Le cas particulier $i = n$ provient du fait que $\Jex_{1}^{(n)} = \Jex_{1\setminus2}^{(n)}$: si
$\minDiv(X^\alpha) = n$, c'est que $\DivSeq(X^\alpha)$ est réduit à $\{n\}$.
\end{proof}

\medskip

Pour un système $\uP$ quelconque, non nécessairement générique, une manière
de raconter l'homogénéité en $P_i$ de $\det W_\calM$  réside dans l'égalité
suivante où $\lambda_1, \dots, \lambda_n \in \bfA$:
$$
\det W_\calM^\sigma(\lambda_1P_1, \dots, \lambda_nP_n) = \pi\, \det W_\calM^\sigma(\uP)
\qquad \hbox {avec} \qquad
\pi = \prod_{X^\alpha \in \calM} \lambda_{\sminDiv(X^\alpha)}
$$
Dans l'énoncé suivant, le système $\uP$ couvre le jeu étalon généralisé
$(p_1X_1^{d_1}, \dots, p_nX_n^{d_n})$ (cf.~\ref{JeuCouvrantEtalon}),
et l'homogénéité est relative aux indéterminées $p_1, \dots, p_n$.
La preuve en est laissée au lecteur.

\begin{prop}[Composante homogène dominante]\leavevmode

Soit $\uP$ un système couvrant le jeu étalon généralisé $\pXD = (p_1X_1^{d_1}, \dots, p_nX_n^{d_n})$.

\begin{enumerate}[\rm i)]
\item 
Soit $\calM$ un sous-module monomial de $\Jex_{1,d}$.

Alors la composante homogène dominante de $\det W_{\calM}(\uP)$ 
est $\det W_{\calM}(\pXD)$ \idest{}
$$
\prod_{X^\alpha \in \calM} p_{\minDiv(X^\alpha)} 
\ = \ 
p_1^{\nu_1} \cdots p_n^{\nu_n}   \qquad
\text{avec $\nu_i = \dim \calM^{(i)}$}
$$  

\item 
La composante homogène dominante de $\omega_{\uP}(X^\emouton)$ 
est $\det W_{1,\delta}(\pXD)$ \idest{}
$$
\prod_{X^\alpha \in \Jex_{1,\delta}} p_{\minDiv(X^\alpha)} 
\ = \ 
p_1^{\nu_1} \cdots p_n^{\nu_n}   \qquad
\text{avec $\nu_i = \dim \Jex_{1,\delta}^{(i)}$}
$$  
\end{enumerate}
\end{prop}


\subsection{Profondeur de l'idéal engendré par les $\big(\det W_{h,d}^\sigma(\protect\uP)\big)_\sigma$}
\label{sectionProfondeurWh}

Notre objectif est d'obtenir, lorsque le système $\uP$ couvre le jeu étalon généralisé $\pXD$
(cf.~\ref{JeuCouvrantEtalon}), 
la minoration suivante de la profondeur :
$$
\Gr\big(\det W^\sigma_{h,d}(\uP),\, \sigma \in \fS_n\big) \,\geqslant\, h
$$
Nous avons vu (dernier point de la proposition \ref{ProprieteWh}) que
l'endomorphisme excédentaire $W_h$ ne dépend pas des $h-1$ derniers
polynômes du système~$\uP$ et l'inégalité ci-dessus précise
donc la dépendance en les $n-(h-1)$ premiers polynômes.

\medskip

\index{idéal!de Veronese}%

Nous allons d'abord examiner le cas particulier du jeu étalon
généralisé $\pXD = (p_1X_1^{d_1}, \dots, p_nX_n^{d_n})$ et notre
anneau habituel $\bfA$ est ici l'anneau de polynômes $\bfR[p_1,
  \ldots, p_n]$.  C'est dans ce cadre que l'idéal de Veronese
$\scrV_{n-(h-1)}$ de $\bfR[p_1, \ldots, p_n]$ joue un rôle
important. On rappelle (cf. le théorème \ref{ProfondeurVeronese} et son
contexte) que, d'une part, on a noté $\scrV_r$ l'idéal monomial de
$\bfR[p_1, \ldots, p_n]$ engendré par les monômes sans facteur carré
de degré~$r$ et que d'autre part, on a montré:
$$
\Gr \big(\scrV_{n-(h-1)}\big) \, \geqslant\, h
$$
Jugeant ce $n-(h-1)$ typographiquement désagréable 
et pour mettre le paramètre $h$ en évidence, 
nous proposons d'utiliser une astuce notationnelle~:
$$
\hderniers 
= \{1,\dots,n-(h-1)\}   \qquad\qquad
\hbox {\footnotesize permettant l'utilisation de $\#\hderniers$ pour $n-(h-1)$}
$$
Enfin, histoire de faire des économies, nous allons pouvoir
remplacer le groupe symétrique $\fS_n$ par un sous-ensemble
beaucoup plus petit, la petitesse étant fonction de $h$. Ainsi, nous
désignons par $\Sigma_h$ n'importe quel sous-ensemble de $\fS_n$ vérifiant:
$$
\hbox {toute partie de $\{1..n\}$ de cardinal $\#\hderniers$ est de la forme $\sigma(\hderniers)$ pour
  un certain $\sigma \in \Sigma_h$}
$$
Puisque par définition on a 
$\{ \sigma(\hderniers),\ \sigma \in \Sigma_h \} = \calP_{\#\hderniers}\big(\{1..n\}\big)$,
on peut écrire:
$$
\scrV_{\#\hderniers} 
\ = \ 
\langle 
\,  p_{\sigma(\hderniers)},\ \sigma \in \Sigma_h \, 
\rangle
\qquad \text{où $p_{\sigma(\hderniers)} \ =\ p_{\sigma(1)} \cdots p_{\sigma(n-h+1)}\ =\ \prod_{i\in \sigma(\hderniers)} p_i$}
$$
Voici des exemples de parties $\Sigma_h$ de $\fS_n$~:
$$
\Sigma_2 = \big\{(n,i), \  1\leqslant i \leqslant n \big\}, 
\qquad\qquad
\Sigma_3 = \big\{(n,i)(n-1,j),\  1 \leqslant i < j \leqslant n \big\}
$$

\label{NOTA05-Sigma2}%
\label{NOTA05-Sigma3}%
\label{NOTA05-Sigmah}%
%
%

\begin{prop} \label{GrWhJeuEtalonGen}
Soit $\pXD$ le jeu étalon généralisé. 
Pour tout $d$, on a 
$$
\Gr \big(
\det W_{h,d}^\sigma(\pXD), \ 
\sigma \in \Sigma_h
\big) 
\ \geqslant \ h
$$
\end{prop}

\begin{proof}
Soit $\sigma \in \fS_n$. 
Pour tout $d$, puisque la matrice de $W_{h,d}^\sigma(\pXD)$ 
est diagonale avec des $p_j$ sur la diagonale, on a~:
$$
\det W_{h,d}^\sigma(\pXD) 
\ = \ 
\prod_{i \in \sigma(\hderniers)} p_i^{\nu_{i,d}}  
$$
où $\nu_{i,d}$ est la dimension du $\bfA$-module 
de base les $X^\alpha \in \Jex_{h,d}$ tels que $\sminDiv(X^\alpha) = i$.
Ainsi il existe un exposant $e$ qui dépend de $h$ et de $d$ tel que 
$$
\big(\scrV_{\#\hderniers}\big)^{e}
\ \subset \ 
\langle 
\det W_{h,d}^\sigma(\pXD), \ 
\sigma \in \Sigma_h
\rangle 
$$
D'après~\ref{ProfondeurVeronese}, on sait que $\Gr( \scrV_{\#\hderniers} ) \geqslant h$.
On termine en utilisant le fait que la profondeur se comporte bien vis-à-vis des puissances 
et de l'inclusion, cf.~\ref{Gr2ProprietesElementaires}.
\end{proof}

\begin{rmq} 
Nous n'avons pas mentionné la dépendance des exposants $\nu_{i,d}$ en $h$ et $\sigma$.
Remarquons que $\nu_{i,d} = 0$ pour $i \notin \sigma(\hderniers)$.
Pour $i \in \sigma(\hderniers)$ fixé, la suite $\nu_{i,d}$ est croissante en $d$ et on a
$\nu_{i,d} \geqslant 1$ pour $d \geqslant d_{\max} = \max\limits_{\#I = h} d_I$.
\end{rmq}

\begin{theo}\label{ProfondeurWh}
Soit $\uP$ un jeu couvrant le jeu étalon généralisé (en particulier la suite générique).
Pour tout $d$, on a 
$$
\Gr \big(
\det W_{h,d}^\sigma(\uP), \ 
\sigma \in \Sigma_h
\big) 
\ \geqslant \ h
$$
\end{theo}

\begin{proof}
Examinons la composante homogène dominante de $\det W_{h,d}^\sigma(\uP)$, polynôme 
de $\bfR[p_1, \dots, p_n]$. 
Un coefficient de $W_{h,d}^\sigma(\uP)$ est de degré $1$ s'il est sur la diagonale, 
et de degré $0$ sinon
(un coefficient de la diagonale est un certain $p_i$ et les coefficients hors de la diagonale 
sont dans $\bfR$).
Par conséquent, dans le développement de ce déterminant, 
un produit indexé par une permutation $\tau$ est, 
dans $\bfR[p_1, \dots, p_n]$,
un monôme de degré le nombre de points fixes de $\tau$.
La composante homogène dominante de $\det W_{h,d}^\sigma(\uP)$ 
est donc réduite au produit correspondant à l'identité. 
C'est exactement $\det W_{h,d}^\sigma(\pXD)$.
On conclut à l'aide de~\ref{GrWhJeuEtalonGen}
et du contrôle de la profondeur par les composantes homogènes dominantes, 
cf.~\ref{ControleProfondeur}.
\end{proof}

\subsection {Note terminologique concernant les idéaux $\Jex_h$ et les endomorphismes~$W_h(\protect\uP)$}

Donner un nom aux objets et attribuer telle notion à tel auteur n'a
pas été facile.  Commençons par les idéaux monomiaux $\Jex_h$. La
terminologie ``monômes repus'' (ceux de $\Jex_1$), ``monômes dodus''
(ceux de~$\Jex_2$) utilisée par Jouanolou (cf. \cite[section
  3.9.1]{J7}) ne nous convenant pas (probablement à cause de la
connotation ``surpoids''), nous avons choisi de ne pas l'utiliser.  Par
ailleurs, nous estimons que les idéaux monomiaux $\Jex_1, \Jex_2$ (qui
n'apparaissent pas chez cet auteur), et également tous les $\Jex_h$
pour $1 \leqslant h \leqslant n$, jouent un rôle structurel non
négligeable.  Nous avons choisi ``idéal $h$-excédentaire'' pour
$\Jex_h$, l'adjectif excédentaire nous semblant plus délicat que ceux
de Jouanolou.

\medskip
En ce qui concerne les matrices et leurs mineurs, c'est un peu plus compliqué.
Nous pensons que tout le monde s'accorde sur la notion de matrice de Sylvester
$\Syl_d = \Syl_d(\uP)$ en degré $d$
$$
\rmK_{1,d}(\uP) = \bfA[\uX]_{d-d_1} \oplus \cdots \oplus  \bfA[\uX]_{d-d_n}  \to  \bfA[\uX]_d,
\qquad (U_1, \dots, U_n) \mapsto U_1 P_1 + \cdots + U_nP_n
$$
Quant aux divers mineurs de $\Syl_d(\uP)$ et tout ce qui gravite
autour, nous avons choisi la lettre~$W$, le choix de cette lettre
étant initialement un acronyme pour weight (poids dans un sens non
mathématique). L'utilisation des matrices $W_{h,d}$, $W_\calM$, qui
sont en fait des \emph {endomorphismes}, a complètement résolu le
problème de signe des mineurs (plus besoin de matrices strictement
carrées), nous a permis un traitement uniforme et des notations
commodes.  Mais comment fallait-il qualifier/baptiser ces diverses
matrices~$W_\calM$? De Macaulay?  Excédentaires?  A notre
connaissance, l'idéal $\Jex_h$ et l'endomorphisme $W_h$ n'existent pas
dans la littérature pour $h \geqslant 3$. Et il faut noter que nous
disposons également d'une version $W^\sigma_{h}$ pour une permutation
$\sigma \in \fS_n$, permutation destinée à modifier l'ordre habituel
de $\{1..n\}$.

\bigskip

Nous rappelons que la notion de matrice strictement carrée à été
précisée dans les premières pages du chapitre~\ref{ChapJeuEtalonGeneralise}
et qu'elle permet de définir sans ambiguité de signe la valeur d'un mineur.

\medskip

Dans \cite[3.9.3.2]{J7}, on voit apparaître, \textit{en tout degré $d$},
deux matrices strictement carrées qui sont les
transposées $\transpose{W_{1,d}(\uP)}$ et $\transpose{W_{2,d}(\uP)}$. 
La formule de Macaulay (faisant intervenir les déterminants de ces matrices en degré 
$d \geqslant \delta+1$) fait l'objet de sa proposition 3.9.4.4. 
La preuve a lieu par récurrence sur le nombre d'indéterminées.
On trouve également cette même preuve ainsi que les ingrédients nécessaires
à sa mise en place, dans la thèse de Chardin \cite[section 4.III]{Chardin4}.
\`A noter que chez Chardin, 
intervient un mineur particulier, baptisé mineur de Macaulay,
mais qui n'est pas directement reliable à nos déterminants
(quand on regarde de près, on comprend qu'il s'agit de 
$\det W_{2,\delta+1}(\uP) \times \Res(P_1, \dots, P_{n-1}, X_n^{d_n})$).

\medskip

Le déterminant $\det W_{1,\delta+1}$ est également visible, au signe
près, chez Cox, Little \& O'Shea dans \cite[chap 3, \S 4, p. 104, def 4.2.]{Cox}
et est noté $D_n$ ; 
on voit apparaître, avant la proposition (4.7), la notation~$D_i$ pour
désigner n'importe quel $\det W^{\sigma_i}_{1,\delta+1}$ où
$\sigma_i \in \fS_n$ vérifie $\sigma_i(n) = i$. L'énoncé (sans
preuve) de la formule de Macaulay, au signe près et en degré $\delta+1$,
figure dans le théorème 4.9 de ce chapitre 3 (au signe près car les
matrices ne sont pas strictement carrées).

\medskip

Dans~\cite[tout début de la section 6]{Demazure1}, Demazure choisit 
de nommer notre $W_{1,d}$ \og \textit{la
matrice de Sylvester (classique pour $d = \delta+1$)}\fg{}, 
alors qu'il s'agit en fait d'une sous-matrice de~$\Syl_{d}$.
Dans un article plus récent~\cite[p. 351]{Demazure2}, il utilise à nouveau la
terminologie de déterminant de Sylvester, qu'il note $D_\sigma$ et qui
n'est autre que notre $\det W^\sigma_{1,\delta+1}$ (à cet endroit de
son article, le résultant n'est pas encore en place).

\bigskip

On ne peut pas dissocier les mineurs de $\Syl_d(\uP)$ de la formule de
Macaulay (celle qui donne le résultant), que nous prouverons plus loin
de deux manières radicalement différentes.  Une première preuve
classique (par récurrence sur $n$) fait l'objet du
théorème~\ref{MacaulayParRecurrence}.  Disons le tout de suite: cette
preuve par récurrence ne nous semble pas du tout élégante et nous lui
préférons la preuve plus structurelle qui figure dans le chapitre~\ref{ChapBW},
même si cette dernière demande plus de travail.

L'approche par récurrence reste mystérieuse malgré le fait
d'avoir traité le cas du jeu étalon généralisé $\pXD = (p_1X_1^{d_1},
\dots, p_nX_n^{d_n})$ dans le
chapitre~\ref{ChapJeuEtalonGeneralise}, chapitre dans lequel on voit
apparaître les idéaux $\Jex_1, \Jex_2$ portant la structure du
numérateur et dénominateur de la formule de Macaulay.

\smallskip
Dans la suite, on retrouvera les objets $W$ à plusieurs reprises, en particulier
dans le chapitre~\ref{ChapWW} (Des identités entre déterminants excédentaires $\det W_\calM(\uP)$)
contenant un certain nombre d'identités entre déterminants de type~$W$.
Nous n'y sommes pas allés avec le dos de la cuillère.

\cleardoublepage

\section{Le jeu circulaire $(X_1^{d_1} - R_1,\dots,X_n^{d_n}-R_n)$ avec $R_i = X_i^{d_i-1}X_{i+1}$}
\label{ChapJeuCirculaire}

Dans ce chapitre, nous allons étudier certains jeux $\uQ$ de la forme $\uQ = \uX^D -
\uR$ où $\uR$ est un jeu monomial de format $D = (d_1,\dots,d_n)$ i.e. chaque $R_i$ est
un monôme de degré $d_i$. Il y a donc équivalence entre se donner le
jeu $\uQ$ ou se donner le jeu $\uR$; il nous arrivera de privilégier
l'un ou l'autre selon le contexte.

\index{jeu!monomial!de format $D$}%

\subsection{Le jeu $\protect\uQ = \big(X_1^{d_1-1}(X_1-X_2),X_2^{d_2-1}(X_2-X_3),\dots,X_n^{d_n-1}(X_n-X_1)\big)=
            \protect\uX^D - \protect\uR$}
\label{IntroJeuCirculaire}

Ici, on souhaite exhiber un système homogène $\uQ = (Q_1, \dots, Q_n)$, à
coefficients dans $\bbZ$, de degré $(d_1, \dots, d_n)$ ayant comme seul zéro
projectif le point $\Un = (1 : 1 : \cdots : 1)$ et tel que ce zéro soit
simple. Ceci dans l'intention d'obtenir le fait que $\det W_{1,\delta}(\uQ)$
est inversible, par exemple $\det W_{1,\delta}(\uQ) = \pm 1$
(c'est la situation opposée de $\det W_{1,d}(\uP) = 0$ pour $d \geqslant \delta+1$ dès 
que les polynômes $P_i$ ont un zéro commun).
\og Géométriquement \fg{}, on voudrait donc que dans chaque carte affine 
$\{X_i = 1\}$, le point affine $(1, 1, \dots, 1)$ soit le seul zéro du système affine
$\uQ(X_i := 1)$ et soit simple. Ce qui se traduit algébriquement  
par l'égalité d'idéaux de $\bfk[X_1,\dots,X_{i-1},X_{i+1},\dots,X_n]$ :
$$
\left \langle \uQ(X_i := 1)\right \rangle 
\ = \ 
\langle X_1-1,\, \dots, \, X_{i-1}-1,\, X_{i+1}-1, \dots, \,X_n-1\rangle
$$
On peut aussi formuler cela de manière globale homogène en imposant 
$\langle \uQ\rangle^\sat = I$ où 
$I$ désigne l'idéal du point $\Un = (1 : 1 : \cdots : 1)$ c'est-à-dire 
$$
I 
\ = \ 
\langle X_i-X_j,\  1 \leqslant i,j \leqslant n \rangle
\ = \ 
\langle X_1-X_2,\  X_2-X_3,\ \dots, \, X_n-X_1 \rangle
$$
Cet idéal $I$ est égal à 
$\{ F \in \bfk[\uX] \ \mid\ F(t,\dots,t) = 0\}$ où $t$ est une indéterminée ;
l'inclusion $\subset$ est claire. 
Dans l'autre sens, en remarquant 
que ${X_i \equiv X_j \bmod I}$, on a 
${F(X_1,X_2,\dots,X_n) - F(X_1,X_1,\dots,X_1) \in I}$ 
pour tout $F \in \bfk[\uX]$ ;
par conséquent, si $F(t,\dots, t) = 0$, alors $F \in I$.
On notera que cette caractérisation de $I$ permet facilement de vérifier 
que l'idéal $I$ est saturé.

\label{NOTA06-UnitPoint}%

\medskip
\og Le \fg{} système $\uQ$ qui vient à l'esprit est :
$$
Q_i(\uX) = X_i^{d_i-1}(X_i - X_{i+1}) 
\qquad 
\hbox {en convenant de $X_{n+1} = X_1$}
$$
Voici une première explication de \og nature géométrique \fg{}.
Du point de vue zéro projectif, la résolution schématique de 
$$
\left\{
\begin {array}{c}
x_1^{d_1-1}(x_1 - x_2) = 0 \\ [0.5em]
x_2^{d_2-1}(x_2 - x_3) = 0 \\
     \vdots                \\
x_n^{d_n-1}(x_n - x_1) = 0 \\
\end {array}
\right.
$$
impose au point $(x_1 : \cdots : x_n)$ d'être égal à $\Un = (1 :
\cdots : 1)$ avec la multiplicité $1$.  En effet, considérons que pour
un point projectif, il y a nécessairement un $x_i$ inversible et
supposons pour simplifier que ce soit~$x_1$; on en
déduit~\mbox{$x_1=x_2$} puis \mbox{$x_2=x_3$}, etc.
Il en serait de même pour un point projectif 
seulement unimodulaire i.e. $1 \in \langle x_1, \cdots,x_n\rangle$,
hypothèse plus faible que l'existence d'un $x_i$ inversible.

Cette \og vérification géométrique \fg{} induit la vérification
algébrique de l'égalité \mbox{$\langle\uQ\rangle^\sat = I$} : dans
chaque localisé en $X_\ell$, on a $I = \langle\uQ\rangle$, d'où un
exposant $e$ tel que \mbox{$X_\ell^e I \subset \langle\uQ\rangle$.}

De manière algébrique plus précise concernant un tel exposant~$e$,
montrons que $X_\ell^\delta I \subset \langle\uQ\rangle$ pour tout~$\ell$. 
Par permutation circulaire, il suffit de le faire pour $\ell = 1$. 
L'identité
$$
X_1^{d_i-1}(X_1 - X_{i+1}) \ =\  
X_1^{d_i-1}(X_1 - X_i) + (X_1^{d_i-1} - X_i^{d_i-1})(X_i - X_{i+1})  + Q_i
$$
montre l'appartenance
$$
X_1^{d_i-1}(X_1 - X_{i+1}) \, \in \, \langle X_1 - X_i, \, Q_i\rangle
$$
En posant $e_i = \sum_{j=1}^i(d_j-1)$, on obtient
$X_1^{e_i}(X_1 - X_{i+1}) \in \langle X_1^{e_{i-1}}(X_1 - X_i),\, Q_i\rangle$ puis par
récurrence sur $i \leqslant n-1$, que :
$$
X_1^{e_i}(X_1 - X_{i+1}) \, \in \, \langle Q_1, \ldots, Q_i\rangle \, \subset \,  \langle\uQ\rangle
$$
Puisque $e_i \leqslant \delta$, on tire $X_1^\delta (X_1 -X_{i+1}) \in \langle\uQ\rangle$ 
pour tout $i \leqslant n-1$, d'où l'inclusion annoncée.

Reste le point le plus important concernant l'inversibilité de $\det W_{1,\delta}(\uQ)$. 
Effectivement, on va voir que ce système convient au sens
où $\det W_{1,\delta}(\uQ) = 1$. 
En réalité, le degré $d = \delta$ est pour $W_{1,d}(\uQ)$ un seuil critique qui partage 
la plage des degrés $d$ en deux intervalles : $d \leqslant \delta$ et $d \geqslant \delta+1$.
Pour $d \geqslant \delta+1$, 
le scalaire $\det W_{1,d}(\uQ)$ est nul car il appartient à l'idéal 
d'élimination, qui est réduit à $0$, vu le zéro commun $\Un = (1 : \cdots : 1)$ du système $\uQ$.
Nous invitons le lecteur à en donner une preuve directe 
(il pourra par exemple remarquer que chaque colonne de $W_{1,d}(\uQ)$ contient exactement 
deux coefficients non nuls qui sont $1$ et $-1$, si bien que la somme des lignes est nulle). 
Quant à la plage $d \leqslant \delta$, elle va nous entraîner dans une 
curieuse combinatoire,
révélant un certain nombre de propriétés remarquables du jeu circulaire 
figurant dans la définition ci-après.


\begin{defn}
Le jeu circulaire $\uQ$ de format $D = (d_1, \dots, d_n)$ est le jeu de polynômes :
$$
\uQ \ = \ \Big(X_1^{d_1 - 1}(X_1 - X_2), \ X_2^{d_2-1}(X_2-X_3), \ 
\dots, \ X_n^{d_n - 1}(X_n - X_1)\Big)
$$
Il est égal à la différence entre le jeu étalon $\uX^D$
et le jeu $\uR = (X_1^{d_1 - 1}X_2, \ X_2^{d_2-1}X_3,\ \dots, \ X_n^{d_n - 1}X_1)$.
\end{defn}

\index{jeu!circulaire}%

Avec ces notations, on a l'égalité d'endomorphismes de Macaulay :
$$
W_{\calM}(\uQ) = \id - W_{\calM}(\uR)
$$

%
%

\subsubsection{But et conséquence !}

En montrant que $\det W_{1,\delta}(\uQ) = 1$ (cf.~\ref{detW1deltaQ}), 
on récupère une propriété très forte 
de $\det W_{1,\delta}(\uP)$ pour la suite générique $\uP$, à savoir que 
ce scalaire est $\bfB_\delta$-régulier. Concrètement, cela dit que pour un polynôme
$F \in \bfA[\uX]_\delta$, on a l'implication 
$$
\det W_{1,\delta}(\uP) \times F \ \in \ \uPdelta 
\ \implies \ 
F \in \uPdelta
$$
Cette propriété de $\bfB_\delta$-régularité est conséquence du fait 
que le jeu circulaire $\uQ$ vérifie $Q_i(\Un) = 0$ pour tout~$i$ 
ce qui permet d'appliquer le résultat~\ref{GenPsatRegScalars} 
reformulé ici :
\begin{quote}
\it 
Soit $\uP$ la suite générique et 
$a \in \bfA = \bfk[\indetsPi]$.
Si la spécialisation de $a$ en le jeu circulaire $\uQ$ vaut $1 \in \bfk$, 
alors $a$ est régulier sur $\bfB_\delta$.
\end{quote}

\medskip

Essayons d'expliquer, pour une suite $\uP$, quel est l'impact du fait que
$\det W_{1,\delta}(\uP)$ soit $\bfB_\delta$-régulier.
Dans l'encadré, et dans l'encadré uniquement, nous tenons à faire figurer~$\uP$
(que nous omettons habituellement) sauf pour $\bfB_\delta = \bfA[\uX]_\delta/\uPdelta$.
Précisons que $\overline\omega : \bfB_\delta \to \bfA$ désigne la
forme linéaire $\omega$ passée au quotient; dire que $\overline\omega$ est
injective, c'est donc dire que $\Ker\omega = \uPdelta$.

\medskip

\fbox{\parbox {0.95\textwidth}{
\centerline {$\displaystyle
\det W_{1,\delta}(\uP) \quad \bfB_\delta\text{-régulier} \quad\overset {(1)}{\Longrightarrow}\quad  
\overline {\omega}_\uP \quad \text{injective} \quad\overset {(2)}{\Longrightarrow}\quad
\overline\omega_{\res,\uP} \quad\text{injective}
$}
\smallskip
$\uP$ régulière:

\noindent
\centerline {$\displaystyle
\overline\omega_{\res,\uP} \quad\text{injective} \quad\overset {(3)}{\Longrightarrow}\quad
\Gr(\omegaRes{\uP}) \ge 2 \quad\overset{(4)}{\Longrightarrow}\quad
\ElimIdeal = \langle \Res(\uP)\rangle
$}

\smallskip
$\uP$ quelconque:

La forme linéaire $\omegaresP : \bfA[\uX]_\delta \to \bfA$
possède la propriété Cramer: $\omegaresP(F)G - \omegaresP(G)F \in \uPdelta$, $\forall\, F, G \in
\bfA[\uX]_\delta$
}}

\bigskip

En ce qui concerne~(1), voir l'implication
\rotatebox[y=1.5mm]{-45}{$\Rightarrow$} dans le diagramme de
l'introduction du chapitre~\ref{ChapBgenerique}.

La forme $\omega$ est multiple de la forme $\omegares$ (définie dans
le chapitre \ref{ChapMacRaeForP}) et l'implication~(2) en
résulte.

L'implication (3) fait l'objet du théorème \ref{omegaresGr2}, cf
également les lemmes \ref{InjectivityToGr2} et~\ref{InjectivityToGr2bis}.

Quant à l'implication (4), elle fait intervenir le résultant défini
dans la section \ref{ChapMacRaeDefResultant}.  Un énoncé plus précis
(que le fait que le résultant soit un générateur de l'idéal
d'élimination) figure dans le théorème~\ref{4pointsPsuperreguliere}.

Le dernier résultat (propriété Cramer) possède un statut un peu
spécial au sens où d'une part, il est absolu i.e. vérifié pour tout
système $\uP$ et d'autre part se spécialise i.e. il suffit de le
prouver pour $\uP$ générique. C'est en ce sens que la $\bfB_\delta$-régularité
de $\det W_{1,\delta}(\uP)$, assurée quand $\uP$ est générique,
participe à l'établissement de cette propriété, cf le
théorème~\ref{omegaresCramer}.

\medskip

Une grande partie de ce chapitre est consacrée à la preuve de l'égalité $a(\uQ) = 1$ 
lorsque $a = \det W_{1,\delta}$.
Nous obtiendrons également des \og bonus\fg{} en montrant que c'est le cas pour $a = \det W_{\calM}$ 
avec $\calM$  sous-module monomial de $\Jex_{1,d}$ pour $d \leqslant \delta$ 
ou bien $\calM$ sous-module monomial de $\Jex_{2,d}$ pour $d \in \bbN$.

\subsubsection{Jeu circulaire généralisé $\pXD + \bsq\,\uR$}

Nous utiliserons souvent le jeu circulaire généralisé $\uQ := \pXD + \bsq\,\uR$:
$$
Q_i \ = \ X_i^{d_i - 1}(p_iX_i + q_iX_{i+1}) 
\ =\ p_i X_i^{d_i} + q_i X_i^{d_i-1}X_{i+1}
$$ 
où $p_i$ et $q_i$ sont des indéterminées.
Ce jeu permet un contrôle plus fin des résultats :
nous allons montrer que $\det W_{1,\delta}(\uQ)$ est un produit de $p_i$.
On revient aisément au jeu circulaire en réalisant $p_i :=1$ et $q_i := -1$.

\label{NOTA06-bsp-bsq}%
\index{jeu!circulaire généralisé}%

\subsubsection{Illustration}

Prenons $D=(1,2,2)$ et considérons  $W_{1,\delta}(\uQ)$ dans des bases monomiales
de $\Jex_{1,\delta}$: à gauche, la base est rangée selon l'ordre lexicographique croissant
(avec $X_1 > X_2 > X_3$) et à droite de manière quelconque:
$$
W_{1,\delta}(\uQ) \ :
\EastBordermatrix{
p_{3} & . & . & . & . & \Heti{X_{3}^{2}} \\ 
. & p_{2} & . & q_{1} & . & \Heti{X_{2}^{2}} \\ 
q_{3} & . & p_{1} & . & . & \Heti{X_{1}X_{3}} \\ 
. & . & . & p_{1} & q_{1} & \Heti{X_{1}X_{2}} \\ 
. & . & . & . & p_{1} & \Heti{X_{1}^{2}} \\ 
}
\qquad\qquad\qquad
\EastBordermatrix{
p_{3} & . & . & . & . & \Heti{X_{3}^{2}} \\ 
q_{3} & p_{1} & . & . & . & \Heti{X_{1}X_{3}} \\ 
. & . & p_{1} & . & q_{1} & \Heti{X_{1}X_{2}} \\ 
. & . & q_{1} & p_{2} & . & \Heti{X_{2}^{2}} \\ 
. & . & . & . & p_{1} & \Heti{X_{1}^{2}} \\ 
}
$$
A priori, il n'est pas immédiat  que le déterminant soit un produit de $p_i$
(en regardant attentivement, on perçoit tout de même une structure en blocs
qui permet de s'en assurer).
En revanche, avec bon rangement monomial, la matrice est triangulaire supérieure
avec une diagonale de $p_i$:
$$
W_{1,\delta}(\uQ) \ : 
\EastBordermatrix{
p_{1} & . & . & q_{3} & . & \Heti{X_{1}X_{3}} \\ 
. & p_{2} & q_{1} & . & . & \Heti{X_{2}^{2}} \\ 
. & . & p_{1} & . & q_{1} & \Heti{X_{1}X_{2}} \\ 
. & . & . & p_{3} & . & \Heti{X_{3}^{2}} \\ 
. & . & . & . & p_{1} & \Heti{X_{1}^{2}} \\ 
}
$$
et le résultat sur le déterminant devient évident.

\noindent
Le but du chapitre est d'exhiber un tel \og rangement monomial triangulaire adéquat\fg{}.

\index{rangement monomial triangulaire}%

\subsection{L'itérateur $w_{\calM}(\protect \uR)$ (piloté par $\minDiv$) et son graphe}

\label{NOTA06-wMiterateur}%
\index{jeu!monomial!de format $D$}%

Soit $\uR = (R_1, \dots, R_n)$ un jeu monomial de format $D$ \idest{} 
$R_i$ est un monôme de degré $d_i$. 
De manière informelle, il lui est associé un \textit{système de ré-écriture}
qui, étant donné $X^\alpha$,
consiste à remplacer $X_i^{d_i}$ par~$R_i$ à condition que
$X_i^{d_i} \mid X^\alpha$ :
$$
X^\alpha \ \longmapsto \ \dfrac{X^\alpha}{X_i^{d_i}} R_i
$$
Ceci est à rapprocher, pour un sous-module monomial $\calM$ de $\bfA[\uX]$,
de l'endomorphisme $W_\calM(\uR)$ de $\calM$ défini par :
$$
W_{\calM}(\uR) : 
\begin{array}[t]{rcl}
\calM & \longrightarrow & \calM \\[0.5em]
X^\alpha 
& \longmapsto &
\pi_{\calM}\Big(\dfrac{X^\alpha}{X_i^{d_i}}R_i \Big) \qquad \text{avec $i = \minDiv(X^\alpha)$} 
\end{array}
$$
On remarquera que l'on a fait le choix du plus petit indice de divisibilité.

\bigskip

Ici, nous allons nous limiter au jeu 
$\uR = (X_1^{d_1-1}X_2, \, X_2^{d_2-1}X_3,\dots, \,X_n^{d_n-1}X_1)$. 
Ainsi, l'endomorphisme $W_{\calM}(\uR)$ a une expression dans laquelle 
il n'y a plus mention apparente de $D$ (le format $D$ est cependant implicite dans
$\minDiv$), à savoir :
$$
X^\alpha 
\ \longmapsto \ 
\pi_{\calM}\Big(\dfrac{X^\alpha}{X_i}X_{i+1}\Big) \qquad \text{avec $i = \minDiv(X^\alpha)$} 
$$

\begin{defn}
Soit $\calM$ un sous-module monomial inclus dans $\Jex_{1,d}$.

\begin{enumerate}[\rm i)]
\item 
L'itérateur relatif à $\calM$ du jeu monomial $\uR$ est 
$$
w_{\calM}(\uR) : 
\begin{array}[t]{rcl}
\text{\rm Monômes}(\bfA[\uX]_d) & \longrightarrow & \text{\rm Monômes}(\bfA[\uX]_d) \\[0.5em]
X^\alpha & \longmapsto &
\begin{cases}
(X^\alpha/X_i^{d_i}) R_i \overset{\rm ici}{=}  
(X^\alpha/X_i) X_{i+1}  & \text{\small si $X^\alpha \in \calM$ 
                                où $i = \minDiv(X^\alpha)$}
\\
X^\alpha & \text{\small sinon}
\end{cases}
\end{array}
$$
En particulier,  pour $d = \delta$ et $\calM = \Jex_{1,\delta}$ :
$$
w_{1,\delta}(\uR) : 
\begin{array}[t]{rcl}
\text{\rm Monômes}(\bfA[\uX]_\delta) & \longrightarrow & \text{\rm Monômes}(\bfA[\uX]_\delta)
\\[0.5em]
X^\alpha 
& \longmapsto &
\begin{cases}
(X^\alpha/X_i)X_{i+1}  & \text{\small si $X^\alpha \in \Jex_{1,\delta}$ 
              où $i = \minDiv(X^\alpha)$}
\\ 
\mouton & \text{\small si $X^\alpha = \mouton$}
\end{cases}
\end{array}
$$

\item 
Le graphe associé à l'itérateur $w_{\calM}(\uR)$ 
est le graphe \textit{simple}
orienté dont les sommets sont les monômes de $\bfA[\uX]_d$, 
avec un arc de $X^\alpha$ vers~$X^\beta$ si $w_\calM(\uR)(X^\alpha) = X^\beta$ ;
en particulier, les points fixes
(qui sont exactement les monômes n'appartenant pas à $\calM$)
sont de degré sortant $0$.
On étiquette chaque arc par l'entier $\minDiv{}$ du monôme source.

\end{enumerate}
\end{defn}

\index{itérateur!$w_\calM$, $w_{1,\delta}$}%

\subsubsection{Quelques exemples de graphe associé à l'itérateur $w_\calM(\uR)$ et 
l'endomorphisme $W_{\calM}(\uR)$}

$\bullet$  
Voilà le graphe de $w_{\calM}(\uR)$ pour $D = (1,1,2)$ et $\calM = \Jex_{2,d}$ avec 
$d=4$ :
$$
\xymatrix{
&&&&X_1^2 X_3^2 \ar[d]^-{1} & \\
&& X_1^3 X_2 \ar[d]^-{1} & & X_1 X_2 X_3^2 \ar[d]^-{1} & \\
&&X_1^2 X_2^2 \ar[d]^-{1} & X_1^2 X_2 X_3 \ar[d]^-{1}  & X_2^2 X_3^2 \ar[d]^-{2} & X_1 X_3^3 \ar[dl]^-{1}  \\
&& X_1 X_2^3 \ar[d]^-{1} & X_1 X_2^2 X_3 \ar[d]^-{1} & X_2 X_3^3 \ar[d]^-{2} &\\
X_1^4 & X_1^3 X_3 & X_2^4 & X_2^3 X_3 & X_3^4 &
}
$$
Les puits sont exactement les monômes qui n'appartiennent pas à $\calM = \Jex_{2,4}$.

\bigskip

$\bullet$ Pour $D=(1,2,2)$ et $\calM = \Jex_{1,\delta}$, on a :
$$
\xymatrix @M=0.4pc @C=1pc @R=1.6pc{
 & X_1^2 \ar[d]^-{1} \\
X_3^2 \ar[d]_-{3} & X_1X_2\ar[d]^-{1}  \\
X_1X_3 \ar[dr]_-{1} & X_2^2 \ar[d]^-{2}  \\
& X_2X_3 = \mouton
}
$$ 
On peut ranger les monômes par strates horizontales (en commençant par exemple par les strates 
proches du \MoutonNoir{}) ; voilà deux tels rangements :
$$
(X_1X_3,\, X_2^2,\, X_1X_2,\, X_3^2, \,X_1^2), \qquad
(X_2^2,\, X_1X_3,\, X_3^2, \, X_1X_2, \, X_1^2)
$$
\index{strate}%
Chacun de ces rangements est un \og rangement monomial triangulaire adéquat\fg{} au sens
matriciel suivant:
$$
W_{1,\delta}(\uQ) \ = 
\EastBordermatrix{
p_{1} & . & . & q_{3} & . & \Heti{X_{1}X_{3}} \\ 
. & p_{2} & q_{1} & . & . & \Heti{X_{2}^{2}} \\ 
. & . & p_{1} & . & q_{1} & \Heti{X_{1}X_{2}} \\ 
. & . & . & p_{3} & . & \Heti{X_{3}^{2}} \\ 
. & . & . & . & p_{1} & \Heti{X_{1}^{2}} \\ 
}
\qquad \qquad 
W_{1,\delta}(\uQ) \ = 
\EastBordermatrix{
p_{2} & . & . & q_1 & . & \Heti{X_{2}^{2}} \\ 
 . & p_{1} & q_3 & . & . & \Heti{X_{1}X_{3}} \\ 
. & . & p_3 & . & . & \Heti{X_{3}^{2}} \\ 
. & . & . & p_1 & q_{1} & \Heti{X_{1}X_{2}} \\ 
. & . & . & . & p_{1} & \Heti{X_{1}^{2}} \\ 
}
$$

\index{rangement monomial triangulaire}%

\bigskip
$\bullet$
Pour comprendre l'importance du puits dans le dernier graphe, 
il est important de montrer un exemple de jeu monomial $\uR$
dont le graphe associé à l'itérateur $w_{\Jex_{1,\delta}}(\uR)$
possède un circuit et n'est donc pas du tout du type précédent.
Par exemple, $D = (1,1,2)$ et $\uR = (X_2, X_1, X_2^2)$ :
$$
\xymatrix {
X_1 \ar@/^7pt/[r]^1 & X_2 \ar@/^7pt/[l]^2 & X_3 = \mouton
}
$$

Voilà le théorème fondamental que nous cherchons à prouver :

\index{théorème!du puits (jeu circulaire)}%

\begin{theo}[Le théorème du puits] \label{TheoPuits}
Toute itération de $w_{1,\delta}(\uR)$ se termine par le \MoutonNoir{}.
Par conséquent, le graphe de $w_{1,\delta}(\uR)$ est une anti-arborescence 
dont le puits est le \MoutonNoir{}.
\end{theo}

\noindent
Il y a un théorème plus général :
\begin{quote}
\it
Soit $\calM$ un sous-module monomial inclus dans $\Jex_{h,d}$.

Pour $(h=1,\, d\leqslant \delta)$ et $(h \geqslant 2,\, d \in \bbN)$, 
toute itération de $w_{\calM}(\uR)$ se termine par un des points fixes du graphe.
Par conséquent, le graphe de $w_{\calM}(\uR)$ est une réunion d'anti-arborescences.
Chaque anti-arborescence possède un unique monôme qui n'est pas dans $\calM$ ;
c'est alors l'unique puits (le point fixe) de l'anti-arborescence.
\end{quote}

La preuve de ce résultat plus général fait l'objet des sections~\ref{SousSectionEnDessousDelta}
et~\ref{SousSectionW2Qtriangulaire}. Ici, nous mettons en place les ingrédients permettant
de prouver le théorème du puits ($\calM = \Jex_{1,\delta}$) qui est le plus important
(et le plus difficile).


\subsubsection{Reformulation en termes d'exposants}

Notons $(\varepsilon_i)_{1 \leqslant i \leqslant n}$ la base canonique de $\bbZ^n$ 
et posons $e_i = (0, \dots, 1, -1, \dots, 0) = \varepsilon_i - \varepsilon_{i+1}$
(avec la convention $\varepsilon_{n+1} = \varepsilon_1$).

L'itérateur attaché à $w_{\calM}(\uR)$ est défini sur 
$\bbN^n_d = \{ \alpha \in \bbN^n \text{\ tel que\ } |\alpha| = d\}$ 
par :
$$
\begin{array}[t]{rcl}
\bbN^n_d & \longrightarrow & \bbN^n_d \\[0.5em]
\alpha & \longmapsto &
\begin{array}[t]{ll}
\alpha - e_i  & \text{\small si $X^\alpha \in \calM$
avec $i = \minDiv(X^\alpha)$}\\ [0.2em]
\alpha & \text{\small sinon}
\end{array}
\end{array}
$$

\label{NOTA06-bbNhmg}%

Nous allons également translater ces exposants par l'exposant du \MoutonNoir{}.
Ceci va avoir pour effet de \og plonger \fg{} l'étude de l'itérateur ci-dessus dans une machinerie 
plus générale, dans laquelle~$D$ et~$d$ disparaissent (au moins visuellement) !
Cette translation se réalise en posant $x = \alpha - (\emouton)$ de sorte que 
$\alpha_i \geqslant d_i \Leftrightarrow x_i > 0$.
Dans cette correspondance, 
l'ensemble $\{i \text{ tel que } x_i > 0\}$ de stricte positivité de~$x$ 
est le pendant de $\DivSeq(X^\alpha)$, et aura donc un rôle important.

\label{NOTA06-emouton}%
\subsubsection{Cas $\calM = \Jex_{1,\delta}$}

En translatant les $\alpha \in \bbN^n_\delta$ par $\emouton = (d_1-1,\cdots,d_n-1)$,
on obtient des $n$-uplets $x = \alpha - (\emouton)$ qui sont dans $\bbZ^n_0$ où 
$$
\bbZ^n_0 \ = \ 
\Big \{ x = (x_1,\dots, x_n) \in \bbZ^n \ \Big\vert \ \sum\limits_{i = 1}^n x_i = 0 
\Big \}
$$
L'itérateur à étudier est donc 
$$
u : 
\begin{array}[t]{rcl}
\bbZ^n_0 & \longrightarrow & \bbZ^n_0 \\[0.5em]
x & \longmapsto &
\begin{array}[t]{ll}
x-e_i &  \text{où $i$ est le plus petit indice tel que $x_i > 0$ si $x \ne \mathbf 0$}  \\ [0.5em]
x & \text{si $x = \mathbf 0$} 
\end{array}
\end{array}
$$
Le théorème~\ref{TheoPuits} se reformule donc ainsi :
\og{\it Toute itération de $u$ se termine par $\mathbf 0$ }\fg{}.
La section suivante est consacrée à la preuve de ce résultat.

\subsection{La chasse aux entiers strictement positifs}

\begin{prop}
On note $(\varepsilon_i)_{1\leqslant i \leqslant n}$ la base canonique de $\bbZ^n$, 
$e_i = \varepsilon_i - \varepsilon_{i+1}$ (cycliquement) et $u$ l'itérateur sur 
$\bbZ^n_0 = \big \{ x = (x_1,\dots, x_n) \in \bbZ^n \mid \sum\limits_{i = 1}^n x_i = 0 \big \}$ 
défini par
$$
u : \ 
x 
\ \longmapsto \ 
\left \{
\begin{array}{ll}
x-e_i &  \text{où $i$ est le plus petit indice tel que $x_i > 0$ si $x \ne \mathbf 0$}  \\ [0.5em]
x & \text{si $x = \mathbf 0$} 
\end{array}
\right .
$$
Alors toute itération de $u$ conduit au point fixe $\mathbf 0$.
\end{prop}

\index{itérateur!chasse aux entiers positifs}%

\begin{proof}
On dispose d'une famille génératrice évidente pour $\bbZ^n_0$, à savoir $\calG = \{e_1, \dots, e_n\}$ 
(et $n$ bases privilégiées $\calB_k = \{ e_1, \dots, e_n \} \setminus \{e_k\}$ de $\bbZ^n_0$).
Faisons deux remarques préliminaires.

$\sbullet$ Si l'on a une écriture $x = a_1 e_1 + \dots + a_n e_n$ avec $a_j \in \bbZ$, alors $x_j = a_j - a_{j-1}$.

$\sbullet$ Si l'on dispose d'une écriture positive de $x$ sur $\calG$ \idest{} $x = b_1 e_1 + \dots + b_n e_n$ 
avec $b_j \geqslant 0$ pour tout $j$, alors $x_i > 0 \Rightarrow b_i > 0$
car $b_i = x_i + b_{i-1}$.

\smallskip
 
Montrons que l'on peut trouver une écriture positive de $x$ sur $\calG$.
On commence par écrire de n'importe quelle façon $x = a_1 e_1 + \dots + a_n e_n$ avec $a_j \in \bbZ$.
Comme $\sum_{j=1}^n e_j = 0$, on voit que $x = (a_1 - m)e_1 + \dots + (a_n - m)e_n$ 
et ceci pour tout $m \in \bbZ$. 
En particulier, en prenant $m = \min a_j$, alors $b_j = a_j - m \geqslant 0$ et 
l'élément $x$ s'écrit sur $\calG$ de manière positive via 
$x = b_1 e_1 + \dots + b_n e_n$.

Nous sommes en mesure d'étudier les itérations de $u$. 
Si $x = \mathbf 0$, alors c'est terminé. Sinon, il existe un~$i$ tel que $x_i > 0$.
En prenant le premier $i$, on a alors 
$u(x) = b_1 e_1 + \cdots + (b_i-1) e_i + \cdots + b_n e_n$.
On voit que l'itéré de $x$ s'écrit \og de la même façon\fg{} que $x$ sur~$\calG$~:
il n'y a qu'une seule composante qui diffère et cette composante diminue d'une unité.
Autrement dit, après une itération, l'entier positif $\sum_{j=1}^n b_j$ diminue d'une unité, 
donc le processus finit par s'arrêter.
\end{proof}

Pour la suite, on va avoir besoin d'exhiber une fonction hauteur.
Celle que l'on a retenue diminue exactement d'une unité à chaque itération de 
$u$.

\index{hauteur!chasse aux entiers positifs}%

\begin{prop} \label{FonctionHauteur}
Considérons la fonction hauteur $\calH_0 : \bbZ^n_0 \rightarrow \bbN$ définie par :
$$
\calH_0(x) \ = \ \sum_{j=1}^n (a_j - m)
\qquad \qquad \text{où } \quad
a_j =  -(x_{j+1} + \dots + x_n)
\quad \text{et} \quad 
m = \min a_j
$$
Pour $x \in \bbZ^n_0$ et n'importe quel indice 
$i$ tel que $x_i>0$ (non nécessairement le premier), on a $\calH_0(x-e_i) = \calH_0(x) -1$.
Ainsi, $\calH_0(x)$ est le nombre d'itérations de $u$ dans la 
proposition précédente.
\end{prop}

\begin{proof}
On reprend les notations de la preuve précédente.
On écrit $x$ sur la base $\calB_n = \{e_1, \dots, e_n\} \setminus \{e_n\}$ de $\bbZ^n_0$.
Donc $x = a_1 e_1 + \cdots + a_n e_n$ avec $a_j \in \bbZ$ et $a_n= 0$. 
D'où l'on tire $a_j = -(x_{j+1} + \cdots + x_n)$.

On en déduit une écriture positive 
$x = b_1 e_1 + \cdots + b_n e_n$ en posant $b_j = a_j - m$ où $m = \min a_j$
(ainsi, si $m = a_k$, on a écrit $x$ sur $\calB_k$ \textit{avec des coefficients positifs}).
En posant $\calH_0(x) = \sum_{j=1}^n b_j$, 
la preuve précédente nous montre que, pour n'importe quel indice $i$ tel que $x_i >0$, 
on a $\calH_0(x-e_i)= \calH_0(x) -1$ 
\end{proof}

\subsection{Structure triangulaire de l'endomorphisme $W_{1,\delta}(\protect \uQ)$}

Dans la section précédente, nous avons prouvé le théorème du puits 
(cf.~\ref{TheoPuits}) :
\begin{quote}
\it
Toute itération de $w_{1,\delta}(\uR)$ se termine par le \MoutonNoir{}.
Par conséquent, le graphe de $w_{1,\delta}(\uR)$ est une anti-arborescence 
dont le puits est le \MoutonNoir{}.
\end{quote}
On peut alors décomposer $\bfA[\uX]_{\delta}$ en strates :
$$
\bfA[\uX]_\delta \ = \ 
\bigoplus_{t \in \bbN} \, \calS_t
$$
où $\calS_t$ désigne le $\bfA$-module libre engendré par les monômes à distance $t$ 
du \MoutonNoir{} ; en particulier $\calS_0 = \bfA X^\emouton$.

\label{NOTA06-strate}%
\index{strate}%

\noindent
Et ce découpage vérifie la propriété :
$$
\text{pour $X^\alpha \in \Jex_{1,\delta}$ à hauteur $t$ du \MoutonNoir{}, 
$\qquad w_{1,\delta}(\uR)(X^\alpha) 
\ \in \ 
\bigoplus_{t' < t} \, \calS_{t'}$
}
$$

\index{hauteur!dans le graphe de l'itérateur!$w_{1,\delta}$}%

\begin{prop} \label{detW1deltaQ}
Relativement à un rangement des monômes de $\Jex_{1,\delta}$ en strates, 
la matrice de $W_{1,\delta}(\uQ)$ est triangulaire \og à diagonale de $p_i$ \fg{}. 

Plus généralement, pour tout sous-module monomial $\calM \subset \Jex_{1,\delta}$, 
et relativement à un rangement en strates, 
la matrice de $W_{\calM}(\uQ)$ est triangulaire \og à diagonale de $p_i$ \fg{}. 

En particulier,  
$$
\det W_{1,\delta}(\uQ) \ = \ 
\prod_{X^\alpha \in \Jex_{1,\delta}}
p_{\minDiv(X^\alpha)}
\qquad \qquad 
\det W_{\calM}(\uQ) \ = \ 
\prod_{X^\alpha \in \calM}
p_{\minDiv(X^\alpha)}
$$
et pour $Q_i = X_i^{d_i} - X_i^{d_i-1}X_{i+1}$, on a $\det W_{1,\delta}(\uQ)  = 1$.
\end{prop}

\index{strate}%
\index{rangement monomial triangulaire}%

\begin{proof}
Relativement à la décomposition $\Jex_{1,\delta} = \bigoplus_{t > 0} \, \calS_{t}$, 
la matrice de $W_{1,\delta}(\uR)$ est triangulaire stricte 
(cf. l'explication donnée avant l'énoncé) ; précisément, 
pour tout $t > 0$, on a 
$$
W_{1,\delta}(\uR)(\calS_t) 
\ \subset \ 
\bigoplus_{t' < t} \, \calS_{t'}
$$
En particulier, pour tout sous-module monomial, $W_{\calM}(\uR)$ est triangulaire stricte 
(il suffit de restreindre l'endomorphisme $W_{1,\delta}(\uR)$ au départ et à l'arrivée à $\calM$).
De plus, la matrice de $W_{\calM}(\pXD)$ du jeu étalon généralisé 
est toujours (quel que soit le rangement monomial) une matrice diagonale de $p_i$.
Finalement, le résultat se déduit de l'égalité 
$W_{\calM}(\uQ) = W_{\calM}(\pXD) + W_{\calM}(\bsq\,\uR)$.
\end{proof}

\subsection{La forme linéaire $\omega_{\protect\uQ}$}

\index{forme linéaire!$\omega=\omega_\uP$ (pilotée par $\minDiv$)}%

On rappelle que $\det W_{1,\delta}(\uP)$ est toujours égal à $\omega_\uP(X^\emouton)$.
Dans le cas du jeu circulaire généralisé~$\uQ$, la formule donnée 
en~\ref{detW1deltaQ} pour $\calM = \Jex_{1,\delta}$ se réécrit donc :
$$
\omega_\uQ(X^\emouton) \ = \ 
\prod_{X^\alpha \in \Jex_{1,\delta}}
p_{\minDiv(X^\alpha)}
$$
En particulier pour $Q_i = X_i^{d_i} - X_i^{d_i-1}X_{i+1}$, on a $\omega_\uQ(X^\emouton) = 1$.
Nous allons montrer que cette 
dernière égalité se généralise à n'importe quel monôme~$X^\alpha$ de degré $\delta$, 
autrement dit $\omega_\uQ(X^\alpha) = 1$.
La stratégie va être d'exhiber une bonne base dans laquelle la matrice 
de~$\Omega_\uQ(X^\alpha)$ --- dont le déterminant vaut $\omega_\uQ(X^\alpha)$ --- 
se décompose en une matrice triangulaire supérieure par blocs, et 
ces blocs (au nombre de 2) ont une forme très particulière rendant le calcul de leur déterminant immédiat.

\bigskip

\begin{prop}\label{OmegaXalpha4blocs}
Soit $X^\alpha \in \bfA[\uX]_\delta$.

\leavevmode
\begin{enumerate}[\rm i)]

\item Notons $\calO_\alpha$ 
le $\bfA$-module \og orbite de~$X^\alpha$ sous~$w = w_{1,\delta}(\uR)$ \fg{} :
$$
\calO_\alpha \ = \ 
\bfA X^\alpha \,\oplus\, \bfA w(X^\alpha) \,\oplus\, \bfA w^2(X^\alpha)
\, \oplus\,  \cdots \, \oplus\, \bfA X^{\emouton}
$$
Désignons par $\overline{\calO_\alpha}$ le supplémentaire monomial de $\calO_\alpha$ 
de sorte que 
$\bfA[\uX]_\delta = \calO_\alpha \, \oplus \, \overline{\calO_\alpha}$.

Pour le jeu circulaire généralisé
(c'est-à-dire $Q_i = p_iX_i^{d_i}+ q_i X_i^{d_i-1}X_{i+1}$), 
on a :
$$
\omega_\uQ(X^\alpha) \ = \ 
\prod_{X^\beta \in \calO_\alpha \setminus  X^\emouton}
-q_{\minDiv(X^\beta)}
\prod_{X^\beta \in \overline{\calO_\alpha}}
p_{\minDiv(X^\beta)}
$$

\item 
Pour le jeu circulaire 
(c'est-à-dire $Q_i = X_i^{d_i} - X_i^{d_i-1}X_{i+1}$), 
on a 
$$
\omega_\uQ(X^\alpha) = 1
$$
\end{enumerate}
\end{prop}

\label{NOTA06-Oalpha}%
%
%

\begin{proof}
La preuve a lieu en trois temps.
On dégage une structure générale pour l'endomorphisme $\Omega_\uP(X^\alpha)$ 
valable pour une suite $\uP$ quelconque, 
puis on examine ce qui se passe pour les jeux $\uR$, $\bsq\,\uR = (q_1R_1, \cdots, q_nR_n)$ 
et enfin pour le jeu circulaire généralisé $\uQ$.

\bigskip

$\rhd$ L'évaluation $\omega_\uP(X^\alpha)$ provient du déterminant d'un certain endomorphisme
dont re-voici la définition :
$$
\Omega_\uP(X^\alpha) : 
\begin{array}[t]{rcl}
\bfA[\uX]_\delta & \longrightarrow & \bfA[\uX]_\delta \\[0.5em]
X^\beta
& \longmapsto &
\begin{array}[t]{ll}
\dfrac{X^\beta}{X_i^{d_i}} P_i  & \text{\small si $X^\beta \in \Jex_{1,\delta}$ 
avec $i = \minDiv(X^\alpha)$}\\ [1.2em]
X^\alpha & \text{\small si $X^\beta = \mouton$}
\end{array}
\end{array}
$$
Pour $\calM \subset \Jex_{1,\delta}$, 
l'endomorphisme $\Omega_\uP(X^\alpha)$ co-restreint à $\calM$ est exactement $W_\calM(\uP)$.
Comme $\overline{\calO_\alpha} \subset \Jex_{1,\delta}$, 
l'endomorphisme $\Omega_\uP(X^\alpha)$ a pour structure,
relativement à la décomposition 
$\bfA[\uX]_\delta = \calO_\alpha \oplus \overline{\calO_{\alpha}}$ :
$$
\Omega_\uP(X^\alpha) 
\ = \ 
\begin{bmatrix}
NO & \star \\
SO & W_\calM(\uP) 
\end{bmatrix}
$$

$\rhd$ Travaillons avec le jeu $\uR$ et déterminons la structure de $\Omega_\uR(X^\alpha)$ : 
nous allons montrer que, pour ce jeu, le bloc NO est bien particulier et le bloc SO est nul.

Comparons les deux endomorphismes :
$$
\Omega_\uR(X^\alpha) : 
\begin{array}[t]{rcl}
\bfA[\uX]_\delta & \longrightarrow & \bfA[\uX]_\delta \\[0.5em]
X^\beta
& \longmapsto &
\begin{array}[t]{l}
\dfrac{X^\beta}{X_i}X_{i+1}
\\ [1.5em]
X^\alpha
\end{array}
\end{array}
\ \text{et } \ 
w_{1,\delta}(\uR) : 
\begin{array}[t]{rcl}
\bfA[\uX]_\delta & \longrightarrow & \bfA[\uX]_\delta \\[0.5em]
X^\beta
& \longmapsto &
\begin{array}[t]{ll}
\dfrac{X^\beta}{X_i} X_{i+1}  & 
\begin{array}{l}
\text{si $X^\beta \in \Jex_{1,\delta}$} \\
\text{avec $i = \minDiv(X^\alpha)$}
\end{array} \\ [1.5em]
X^\beta & \text{\small si $X^\beta = \mouton$}
\end{array}
\end{array}
$$
On constate que ces deux endomorphismes sont quasiment identiques : 
seule l'image de $X^\emouton$ diffère.

\noindent
Déterminons à présent $\Omega_\uR(X^\alpha)$ sur $\calO_\alpha$.
Notons 
$\ell = \dim \calO_\alpha$ de sorte que $w^{\ell-1}(X^\alpha) = X^\emouton$.
Soit $X^\beta \in \calO_\alpha$ qui s'écrit donc $w^e(X^\alpha)$ ; 
on a 
$$
\Omega_\uR(X^\alpha)\Big(w^e(X^\alpha)\Big) 
= 
\left\{
\begin{array}{ll}
X^\alpha & \text{si $e = \ell-1$} \\
w^{e+1}(X^\alpha) & \text{si $e < \ell-1$}
\end{array}
\right.
$$

\noindent
En raisonnant modulo $\ell$ sur l'exposant $e$, on a
donc l'égalité 
$\Omega_\uR(X^\alpha)\Big(w^e(X^\alpha)\Big) 
 = w^{e+1}(X^\alpha)$.

\noindent
Ainsi $\Omega_\uR(X^\alpha)$ 
laisse stable $\calO_\alpha$ avec une structure \og de cycle \fg{}.
On en déduit la structure des endomorphismes $\Omega_\uR(X^\alpha)$ et  
$\Omega_{\bsq\uR}(X^\alpha)$ co-restreint à $\calO_\alpha$ :
$$
\EastBordermatrix{
0 &  &  &  &  & 1 & \Heti{X^\alpha} \\ 
1 & 0 &  & &  & 0 & \Heti{w(X^\alpha)}\\
  & 1 & \ddots &  &  & & \Heti{\vdots} \\
  &  & \ddots & \ddots &  & \vdots & \Heti{\vdots} \\
  &  &  & \ddots & 0 & & \Heti{w^{\ell-2}(X^\alpha)} \\
  &  &  &  & 1 & 0 & \Heti{X^\emouton} 
}
\hspace{3cm}
\EastBordermatrix{
0 &  &  &  &  & 1 & \Heti{X^\alpha} \\ 
q_{i_0} & 0 &  & &  & 0 & \Heti{w(X^\alpha)}\\
  & q_{i_1} & \ddots &  &  & & \Heti{\vdots} \\
  &  & \ddots & \ddots &  & \vdots & \Heti{\vdots} \\
  &  &  & \ddots & 0 & & \Heti{w^{\ell-2}(X^\alpha)} \\
  &  &  &  & q_{i_{\ell - 2}} & 0 & \Heti{X^\emouton} 
}
$$
où $i_k = \minDiv\big(w^k(X^\alpha)\big)$.

\bigskip

$\rhd$ Revenons au jeu $\uQ = \pXD + \bsq\,\uR$ généralisé. 

\noindent
L'endomorphisme $\Omega_\uQ(X^\alpha)$ est construit à partir de 
l'endomorphisme $\Omega_\uR(X^\alpha)$ en \og ajoutant \fg{} ce qui correspond au jeu étalon, 
à savoir une diagonale ayant 
\begin{itemize}
\item un $0$, à la ligne du \MoutonNoir{} 
\item à la ligne de $X^\alpha$, $p_i$ où $i = \minDiv(X^\alpha)$ 
\end{itemize}
Avec ce qui précède, on obtient alors :
$$
\Omega_\uQ(X^\alpha) \ = \ 
\begin{bmatrix}
NO & \star \\
\bf 0 & SE 
\end{bmatrix}
$$
où 
$$
NO \ = \ 
\EastBordermatrix{
\mathstrut p_{i_0} &  &  &  &  & 1 & \Heti{X^\alpha} \\ 
q_{i_0} & p_{i_1} &  & &  & 0 & \Heti{w(X^\alpha)}\\
  & q_{i_1} & \ddots &  &  & & \Heti{\vdots} \\
  &  & \ddots & \ddots &  & \vdots & \Heti{\vdots} \\
  &  &  & \ddots & p_{i_{\ell-2}} & & \Heti{w^{\ell-2}(X^\alpha)} \\
  &  &  &  & q_{i_{\ell - 2}} & 0 & \Heti{X^\emouton} 
}
\qquad 
\text{et}
\qquad 
SE \ = \ 
W_{\overline{\calO_\alpha}}(\uQ)
$$
où $i_k = \minDiv\big(w^k(X^\alpha)\big)$.

\bigskip

$\rhd$ On en déduit le déterminant de $\Omega_\uQ(X^\alpha)$ qui vaut :
$$
\omega_\uQ(X^\alpha) 
\ = \ 
\det NO \, \times \, \det W_{\overline{\calO_\alpha}}(\uQ)
$$
avec 
$\det NO = 
(-1)^{\text{taille NO}-1} 
\prod\limits_{X^\beta \in \calO_\alpha \setminus  X^\emouton}
q_{\minDiv(X^\beta)}
$, qui s'écrit encore 
$\prod\limits_{X^\beta \in \calO_\alpha \setminus  X^\emouton}
-q_{\minDiv(X^\beta)}$.

\bigskip
\noindent
Quant à la valeur de 
$\det W_{\overline {O_\alpha}}(\uQ)$, elle est fournie par
l'utilisation de~\ref{detW1deltaQ} 
avec $\calM = \overline{O_\alpha}$, sous-module de $\Jex_{1,\delta}$. 
\end{proof}

\begin{exemple}\label{Graphe122}
Rien qu'à la lecture du graphe, on peut donc déterminer $\omega_\uQ(X^\alpha)$ :
partir de $X^\alpha$ et descendre jusqu'au puits en prenant le produit 
des $-q_{\text{étiquette arc}}$ puis multiplier par le produit des $p_{\text{étiquette autres arcs}}$.

Illustrons le sur l'exemple $D=(1,2,2)$. 
Pour $X^\alpha = X_1X_2$, on trouve
$\omega_\uQ(X_1X_2) = (-q_1)( -q_2) p_1 p_3 p_1$.

$$
\vcenter{
\xymatrix @M=0.4pc @C=1pc @R=1.4pc{
 & X_1^2 \ar[d]^-{1} \\
X_3^2 \ar[d]_-{3} & X_1X_2\ar[d]^-{1}  \\
X_1X_3 \ar[dr]_-{1} & X_2^2 \ar[d]^-{2}  \\
& X_2X_3 
}}
\qquad\qquad
\begin{array}{lcl}
\omega_\uQ(X_1^2) \ =\  -q_1^2\,q_2\,p_1\,p_3
& &
\omega_\uQ(X_1X_2) \ =\  q_1\,q_2 \,p_1^2\,p_3 
\\ [0.6cm]
\omega_\uQ(X_1X_3) \ =\  -q_1\,p_1^2 \,p_2\,p_3 
& &
\omega_\uQ(X_2^2) \ =\  -q_2\,p_1^3\,p_3
\\ [0.6cm]
\omega_\uQ(X_2X_3) \ =\  p_1^3\,p_2\,p_3 
& & 
\omega_\uQ(X_3^2) \ =\  q_1 \,q_3\,p_1^2\,p_2 
\end{array}
$$
\medskip

Ceci est bien sûr confirmé par la structure des matrices des $\Omega_\uQ(X^\alpha)$.
$$
\Omega_\uQ(X_1^2) \ = \ 
\EastBordermatrix{
p_{1} & . & . & 1 & \VR . & . & \Heti{X_{1}^{2}} \\ 
q_{1} & p_{1} & . & . & \VR . & . & \Heti{X_{1}X_{2}} \\ 
. & q_{1} & p_{2} & . & \VR . & . & \Heti{X_{2}^{2}} \\ 
. & . & q_{2} & . & \VR q_{1} & . & \Heti{X_{2}X_{3}} \\ 
\HR{6} 
. & . & . & . & \VR p_{1} & q_{3} & \Heti{X_{1}X_{3}} \\ 
. & . & . & . & \VR . & p_{3} & \Heti{X_{3}^{2}} \\ 
\noalign{\vskip-1pt}
}
\qquad \qquad
\Omega_\uQ(X_1X_2) \ = \ 
\EastBordermatrix{
p_{1} & . & 1 & \VR . & . & q_{1} & \Heti{X_{1}X_{2}} \\ 
q_{1} & p_{2} & . & \VR . & . & . & \Heti{X_{2}^{2}} \\ 
. & q_{2} & . & \VR q_{1} & . & . & \Heti{X_{2}X_{3}} \\ 
\HR{6} 
. & . & . & \VR p_{1} & q_{3} & . & \Heti{X_{1}X_{3}} \\ 
. & . & . & \VR . & p_{3} & . & \Heti{X_{3}^{2}} \\ 
. & . & . & \VR . & . & p_{1} & \Heti{X_{1}^{2}} \\ 
\noalign{\vskip-1pt}
}
$$
$$
\Omega_\uQ(X_1X_3) \ = \ 
\EastBordermatrix{
p_{1} & 1 & \VR . & . & q_{3} & . & \Heti{X_{1}X_{3}} \\ 
q_{1} & . & \VR q_{2} & . & . & . & \Heti{X_{2}X_{3}} \\ 
\HR{6} 
. & . & \VR p_{2} & q_{1} & . & . & \Heti{X_{2}^{2}} \\ 
. & . & \VR . & p_{1} & . & q_{1} & \Heti{X_{1}X_{2}} \\ 
. & . & \VR . & . & p_{3} & . & \Heti{X_{3}^{2}} \\ 
. & . & \VR . & . & . & p_{1} & \Heti{X_{1}^{2}} \\ 
\noalign{\vskip-1pt}
}
\qquad \qquad
\Omega_\uQ(X_2^2) \ = \ 
\EastBordermatrix{
p_{2} & 1 & \VR . & q_{1} & . & . & \Heti{X_{2}^{2}} \\ 
q_{2} & . & \VR q_{1} & . & . & . & \Heti{X_{2}X_{3}} \\ 
\HR{6} 
. & . & \VR p_{1} & . & q_{3} & . & \Heti{X_{1}X_{3}} \\ 
. & . & \VR . & p_{1} & . & q_{1} & \Heti{X_{1}X_{2}} \\ 
. & . & \VR . & . & p_{3} & . & \Heti{X_{3}^{2}} \\ 
. & . & \VR . & . & . & p_{1} & \Heti{X_{1}^{2}} \\ 
\noalign{\vskip-1pt}
}
$$
$$
\Omega_\uQ(X_2X_3) \ = \ 
\EastBordermatrix{
1 & \VR q_{1} & q_{2} & . & . & . & \Heti{X_{2}X_{3}} \\ 
\HR{6} 
. & \VR p_{1} & . & . & q_{3} & . & \Heti{X_{1}X_{3}} \\ 
. & \VR . & p_{2} & q_{1} & . & . & \Heti{X_{2}^{2}} \\ 
. & \VR . & . & p_{1} & . & q_{1} & \Heti{X_{1}X_{2}} \\ 
. & \VR . & . & . & p_{3} & . & \Heti{X_{3}^{2}} \\ 
. & \VR . & . & . & . & p_{1} & \Heti{X_{1}^{2}} \\ 
\noalign{\vskip-1pt}
}
\qquad \qquad
\Omega_\uQ(X_3^2) \ = \ 
\EastBordermatrix{
p_{3} & . & 1 & \VR . & . & . & \Heti{X_{3}^{2}} \\ 
q_{3} & p_{1} & . & \VR . & . & . & \Heti{X_{1}X_{3}} \\ 
. & q_{1} & . & \VR q_{2} & . & . & \Heti{X_{2}X_{3}} \\ 
\HR{6} 
. & . & . & \VR p_{2} & q_{1} & . & \Heti{X_{2}^{2}} \\ 
. & . & . & \VR . & p_{1} & q_{1} & \Heti{X_{1}X_{2}} \\ 
. & . & . & \VR . & . & p_{1} & \Heti{X_{1}^{2}} \\ 
\noalign{\vskip-1pt}
}
$$
\hfill $\square$
\end{exemple}

\subsubsection{La forme linéaire $\omega_\uQ$ pour le jeu circulaire}

Pour le jeu circulaire $\uQ$ défini par $Q_i = X_i^{d_i-1}(X_i - X_{i+1})$, 
on a donc obtenu $\omega_\uQ(X^\alpha) = 1$
pour tout monôme $X^\alpha$ de degré $\delta$. 
Ceci prouve que la forme linéaire $\omega_\uQ : \bfA[\uX]_{\delta} \to \bfA$ 
est la restriction à $\bfA[\uX]_\delta$ de l'évaluation 
$F \mapsto F(1,\dots,1)$ au point $(1 : \cdots : 1)$.
Cette forme linéaire ne dépend donc que de $\delta$ et pas du format $D$. 
Par conséquent, il en est de même de son noyau.
L'objet de la proposition suivante est de redémontrer ce résultat sans disposer de la structure
de l'endomorphisme $\Omega_\uQ(X^\alpha)$.

\begin{prop} \label{ProprietesJeuCirculaire}
Le jeu circulaire $\uQ$ défini par $Q_i = X_i^{d_i-1}(X_i - X_{i+1})$ possède les propriétés suivantes.

\begin{enumerate}[\rm i)]
\item
Pour tout $X^\alpha$ de degré $\delta$, on a $X^\alpha - X^{\emouton} \in \langle\uQ\rangle_\delta$.

\item
Pour tout $X^\alpha$ de degré $\delta$, on a $\omega_\uQ(X^\alpha) = 1$.
Autrement dit, 
$\omega_\uQ = \Big(\displaystyle \sum_{|\alpha|=\delta} X^\alpha \Big)^\star$.

\item
On a les égalités
$$
\Ker \omega_\uQ 
\ =\ 
\bigoplus_{X^\alpha \in \Jex_{1,\delta}} \bfA \, (X^\alpha - X^{\emouton})
\ =\ 
\langle\uQ\rangle_\delta
$$
En particulier, 
$$
\bfA[\uX]_\delta 
\ = \ 
\langle\uQ\rangle_\delta
\  \oplus \ 
\bfA X^{\emouton}
$$
\end{enumerate}
\end{prop}

\begin{proof}
i) 
C'est évident pour $X^\alpha = X^\emouton$. Prenons désormais $X^\alpha \in \Jex_{1,\delta}$.
L'endomorphisme $W_{1,\delta}(\uQ) = \id - W_{1,\delta}(\uR)$ envoie $X^\alpha$ dans 
$\langle \uQ \rangle_\delta$. 
Notons $w=W_{1,\delta}(\uR)$ et raisonnons modulo $\langle \uQ \rangle_\delta$. 
On a alors $X^\alpha \equiv w(X^\alpha)$ puis 
$w(X^\alpha) \equiv w^2(X^\alpha)$ 
et ainsi de suite, jusqu'à tomber sur le \MoutonNoir{} 
d'après~\ref{TheoPuits}.
D'où le résultat.

ii) D'après~\ref{detW1deltaQ}, on a $\det W_{1,\delta}(\uQ) = 1$, c'est-à-dire 
$\omega_\uQ(X^{\emouton}) = 1$.
De plus, on sait que $\omega_\uQ$ est nulle sur $\langle\uQ\rangle_\delta$ 
(propriété générale de la forme linéaire $\omega$, cf~\ref{omegaProprietes}).
On a donc $\omega_\uQ(X^\alpha) = \omega_\uQ(X^{\emouton}) = 1$ grâce à i).

iii) 
On a les inclusions :
$$
\bigoplus_{X^\alpha \in \Jex_{1,\delta}} \bfA \, (X^\alpha - X^{\emouton})
\quad \subset \quad
\langle \uQ \rangle_\delta 
\quad \subset \quad
\Ker \omega_\uQ
$$
Utilisons maintenant le fait que $\omega_\uQ(X^\alpha) = 1$ pour tout $| \alpha | = \delta$. 
Dans la base duale de la base monomiale de $\bfA[\uX]_\delta = 
\Jex_{1,\delta} \oplus  \bfA X^{\emouton}$, 
la forme linéaire $\omega_\uQ$ a donc toutes ses composantes égales à~1. 
Son noyau est donc libre avec pour base la famille des 
$X^\alpha-X^{\emouton}$ avec $X^\alpha \in \Jex_{1,\delta}$. 
Les inclusions ci-dessus sont donc des égalités.
\end{proof}

\begin{rmq}

\leavevmode

$\bullet$
Ce résultat permet de retrouver l'inclusion difficile 
$I \subset \langle \uQ \rangle^\sat$ donnée dans l'introduction 
(cf.~page~\pageref{IntroJeuCirculaire}), du moins pour $\delta \geqslant 1$.
En prenant 
$X^\alpha = X^\gamma X_i$ et $X^\beta = X^\gamma X_j$, on a :
$$
\forall\, |\gamma| = \delta-1,\quad 
X^\gamma (X_i - X_j) \in \langle \uQ \rangle 
$$
Ceci est même plus précis : ici on a le degré $\delta-1$ pour les $X^\gamma$
alors qu'à la page~\pageref{IntroJeuCirculaire} on avait seulement~$\delta$.

\bigskip

$\bullet$ 
En ce qui concerne le point iii), voici un autre argument pour montrer directement 
l'inclusion $\langle \uQ \rangle_\delta \supset \Ker \omega_\uQ$.
Un argument analogue sera utilisé pour le système générique $\uP$ dans 
la preuve de~\ref{omegaresInjectiveDoncCramer}.
Présentons-le dans le cas du système circulaire $\uQ$, qui n'est pas du tout générique.
Soit $F \in \bfA[\uX]_\delta$ tel que $\omega_\uQ(F) = 0$.
D'après l'appartenance à la Cramer,
pour tout $G \in \bfA[\uX]_\delta$, on~a  
$\omega_\uQ(G) F - \omega_\uQ(F) G \in \langle \uQ \rangle_\delta$ ce
qui se résume ici à $\omega_\uQ(G) F \in \langle \uQ \rangle_\delta$. 
Prenons $G = X^{\emouton}$. On~a 
$\omega_\uQ(X^\emouton) = \det W_{1,\delta}(\uQ) = 1$. 
D'où $F  \in \langle \uQ \rangle_\delta$.
\end{rmq}

\subsubsection*{Le jeu de Chardin}

Soit $\uQ'$ un système dont l'idéal saturé 
$\langle \uQ'\rangle^\sat$ est l'idéal du point $\Un$. 
On se gardera de croire que $\uQ'$ possède les propriétés du jeu circulaire. 
En voici un exemple avec le format $D = (2,2,3)$.

Considérons le triplet en $(Y_1, Y_2, Y_3)$ :
$$
(Y_1Y_3 - Y_2^2,\quad  -Y_1^2 + Y_2Y_3,\quad Y_1Y_3^2)
\leqno (\star)
$$
et posons :
$$
Y_1 = X_1-X_3, \quad Y_2=X_2-X_3,\quad Y_3=X_3
$$
On obtient ainsi un autre jeu $\uQ'(\uX) = (Q'_1, Q'_2, Q'_3)$ avec 
$Q'_3 = (X_1-X_3)X_3^2$. 
On laisse le soin au lecteur de vérifier que le saturé de l'idéal engendré par les 
polynômes de $(\star)$ est l'idéal $\langle Y_1,Y_2\rangle$, idéal du point~$(0:0:1)$. 
Par conséquent, le saturé de l'idéal $\langle \uQ'\rangle$ est l'idéal 
$\langle X_1-X_3, X_2-X_3\rangle$ de notre point fétiche $(1:1:1)$.  
Qu'en est-il de $\det W_{1,\delta}(\uQ')$ ? 
Il est nul ! Pire, $W_{2,\delta}(\uQ')$ est la matrice nulle $1 \times 1$ 
comme on le vérifie facilement. 
On a $\delta = 4$, $\Jex_{2,\delta} = \bfA X_1^2 X_2^2$ et $\minDiv(X_1^2X_2^2) = 1$ ;
enfin le coefficient en~$X_1^{d_1}$ de~$Q'_1$ est nul. 
La nullité de $\det W_{2,\delta}(\uQ')$ a pour conséquence la
nullité de la forme linéaire~$\omega_{\uQ'}$. 
On pourrait d'ailleurs vérifier cette dernière assertion par la force brute : 
la base monomiale de $\bfA[X_1,X_2,X_3]_4$ est constituée de $15$ monômes~$X^\alpha$ et
toutes les matrices $\Omega_{\uQ'}(X^\alpha)$ ont leur première ligne
nulle, à l'exception de $\Omega_{\uQ'}(X_1^4)$ qui est singulière pour une
autre raison.

Le système $(\star)$ est extrait d'une famille de systèmes au point
$(0 : \cdots : 0 : 1$) utilisée par Marc Chardin dans l'étude des
sous-résultants~\cite{Chardin1}. 
La plupart des systèmes de cette famille conduisent
à un jeu $\uQ'$ avec $\omega_{\uQ'} = 0$, donc déficient par rapport
à la propriété convoitée dans ce chapitre 
(\og la plupart \fg{} ayant un sens bien précis: les seuls cas pour lesquels 
$\det W_{2,\delta}(\uQ') \ne 0$ sont $n=2$ et $n=3$ avec $D$ de la forme $(d, d, 2)$).

\subsection{Un premier bonus : l'endomorphisme $W_{\calM}(\protect\uQ)$ pour
            $\calM\subset\Jex_{1,d}$ et $d \leqslant \delta$}
\label{SousSectionEnDessousDelta}

Nous continuons à désigner par $\varepsilon_i$ le $i$-ème vecteur de la base canonique de $\bbZ^n$
et $e_i = \varepsilon_i - \varepsilon_{i+1}$ avec $\varepsilon_{n+1} = \varepsilon_1$.

\begin{prop} \label{Chasseh=1}
Soit $d \leqslant \delta$.
Considérons l'itérateur $u$ sur $\bbZ^n_{d-\delta}
\ = \ 
\big \{ x \in \bbZ^n \mid  |x| = d - \delta \big \}$ 
$$
u : \ 
x 
\ \longmapsto \ 
\left \{
\begin{array}{ll}
x-e_i &  \text{où $i$ est le plus petit indice tel que $x_i > 0$ si cela existe}  \\ [0.5em]
x & \text{sinon} 
\end{array}
\right .
$$
dont les points fixes sont les $n$-uplets de $\bbZ^n_{d-\delta}$ n'ayant aucune composante 
strictement positive.
\noindent
Alors toute itération de $u$ conduit à un point fixe. 
\end{prop}

\begin{proof}
Fixons $\ell$ quelconque. Associons à $x \in \bbZ^n_{d-\delta}$ l'élément 
$y \in \bbZ^n_{0}$ défini par $y = x - s \varepsilon_\ell$ où $s = |x| = d-\delta \leqslant 0$.
On a ou bien $y_j = x_j - s$ ou bien $y_j = x_j$ selon que $j$ est égal à $\ell$ ou pas.
Donc, dans tous les cas, puisque $s \leqslant 0$, on a
l'implication $x_i > 0 \Rightarrow y_i > 0$.
Or pour $y \in \bbZ^n_0$ et $y_i > 0$, la proposition~\ref{FonctionHauteur} dit que 
$\calH_0(y -e_i) = \calH_0(y)-1 \in \bbN$.
Ainsi l'entier $\calH_0(y)$ diminue d'une unité après chaque itération de~$u$.
Par conséquent, le processus finit par s'arrêter sur un point fixe de $u$.
\end{proof}

On en déduit alors le résultat suivant :
\begin{quote}
\it 
Pour $d \leqslant \delta$, le graphe de l'itérateur 
$w_{1,d}(\uR)$ est une réunion d'anti-arborescences. Chaque anti-arborescence possède un unique 
monôme qui n'est pas dans $\Jex_{1,d}$, qui est alors l'unique puits (anti-racine).
\end{quote}
On peut alors décomposer $\bfA[\uX]_{d}$ en strates :
$$
\bfA[\uX]_d \ = \ 
\bigoplus_{t \in \bbN} \, \calS_t
$$
où $\calS_t$ désigne le $\bfA$-module libre engendré par les monômes à distance $t$ 
de leur anti-racine ;
en particulier $\calS_0 = \bigoplus_{X^\alpha \notin \Jex_{1,d}} \bfA X^\alpha$.
\noindent
Et ce découpage vérifie la propriété :
$$
\text{pour $X^\alpha \in \Jex_{1,d}$ à hauteur $t$ de son anti-racine, on a  
$\qquad w_{1,d}(\uR)(X^\alpha) 
\ \in \ 
\bigoplus_{t' < t} \, \calS_{t'}$
}
$$

\index{strate}%
\index{hauteur!dans le graphe de l'itérateur!$w_{1,d}$}%

De la même façon qu'en~\ref{detW1deltaQ}, on en déduit le résultat suivant :

\begin{prop}

Soit $d \leqslant \delta$.
Relativement à un rangement des monômes de $\Jex_{1,d}$ en strates, 
la matrice de $W_{1,d}(\uQ)$ est triangulaire \og à diagonale de $p_i$ \fg{}. 

Plus généralement, pour tout sous-module monomial $\calM \subset \Jex_{1,d}$, 
et relativement à un rangement en strates, 
la matrice de $W_{\calM}(\uQ)$ est triangulaire \og à diagonale de $p_i$ \fg{}.

En particulier,  
$$
\det W_{1,d}(\uQ) \ = \ 
\prod_{X^\alpha \in \Jex_{1,d}}
p_{\minDiv(X^\alpha)}
$$
et pour $Q_i = X_i^{d_i} - X_i^{d_i-1}X_{i+1}$, on a $\det W_{1,d}(\uQ)  = 1$.
\end{prop}

\index{strate}%

\subsubsection{Exemple}

Prenons $D = (4,2,1)$, et $d = 3 = \delta - 1$. 
\`A gauche, le graphe orienté de~$w_{1,d}(\uR)$ et à droite la matrice de $W_{1,d}(\uQ)$.

$$
\vcenter{
\xymatrix @M=0.4pc @C=1pc @R=1.4pc{
& X_2^3 \ar[d]^-{2} & \\
X_3^3 \ar[d]^-{3} & X_2^2 X_3 \ar[d]^-{2} & \\
X_1X_3^2 \ar[d]^-{3} & X_2 X_3^2\ar[d]^-{3} & X_1X_2^2  \ar[ld]^-{2} \\
X_1^2 X_3 \ar[d]^-{3} & X_1X_2X_3 \ar[d]^-{3} & \\
X_1^3 & X_1^2 X_2 & \\
}}
\hspace{1cm}
W_{1,\delta-1}(\uQ)
\ = \ 
\EastBordermatrix{
p_{3} & . & . & q_{3} & . & . & . & . & \Heti{X_{1}^{2}X_{3}} \\ 
. & p_{3} & q_{2} & . & q_{3} & . & . & . & \Heti{X_{1}X_{2}X_{3}} \\ 
. & . & p_{2} & . & . & . & . & . & \Heti{X_{1}X_{2}^{2}} \\ 
. & . & . & p_{3} & . & . & q_{3} & . & \Heti{X_{1}X_{3}^{2}} \\ 
. & . & . & . & p_{3} & q_{2} & . & . & \Heti{X_{2}X_{3}^{2}} \\ 
. & . & . & . & . & p_{2} & . & q_{2} & \Heti{X_{2}^{2}X_{3}} \\ 
. & . & . & . & . & . & p_{3} & . & \Heti{X_{3}^{3}} \\ 
. & . & . & . & . & . & . & p_{2} & \Heti{X_{2}^{3}} \\ 
}
$$
%

Une remarque en passant : 
comme l'itérateur laisse stable chaque composante connexe, on peut 
également procéder composante connexe par composante connexe.
Pour ce même exemple, on a aussi :
$$
W_{1,\delta-1}(\uQ)
\ = \ 
\EastBordermatrix{
p_{3} & q_{3} & . & \VR . & . & . & . & . & \Heti{X_{1}^{2}X_{3}} \\ 
. & p_{3} & q_{3} & \VR . & . & . & . & . & \Heti{X_{1}X_{3}^{2}} \\ 
. & . & p_{3} & \VR . & . & . & . & . & \Heti{X_{3}^{3}} \\ 
\HR{8} 
. & . & . & \VR p_{3} & q_{2} & q_{3} & . & . & \Heti{X_{1}X_{2}X_{3}} \\ 
. & . & . & \VR . & p_{2} & . & . & . & \Heti{X_{1}X_{2}^{2}} \\ 
. & . & . & \VR . & . & p_{3} & q_{2} & . & \Heti{X_{2}X_{3}^{2}} \\ 
. & . & . & \VR . & . & . & p_{2} & q_{2} & \Heti{X_{2}^{2}X_{3}} \\ 
. & . & . & \VR . & . & . & . & p_{2} & \Heti{X_{2}^{3}} \\ 
\noalign{\vskip-1pt}
}
$$

\subsection{Second bonus : l'endomorphisme $W_{\calM}(\protect\uQ)$ pour $\calM \subset \Jex_{2,d}$ et $d \in \bbN$}
\label{SousSectionW2Qtriangulaire}

\subsubsection{$\bullet$ A l'aide de l'itérateur translaté}

Ce cas a déjà été en partie traité puisque $\Jex_{2,d} \subset \Jex_{1,d}$.
Il ne reste donc plus que le cas $\Jex_{2,d}$ avec $d \geqslant \delta+1$ à examiner.
On peut encore procéder à l'aide du résultat \og chasse aux entiers positifs \fg{} établi 
pour $\Jex_{1,\delta}$ et l'adapter. Ici, la condition 
\textit{le plus petit $i$ tel que} \dots{} est importante, comme on pourra 
s'en convaincre en lisant la preuve.

\begin{prop} \label{Chasseh=2}
Soit $d \in \bbN$. 
Considérons l'itérateur $u$ sur $\bbZ^n_{d-\delta}
\ = \ 
\big \{ x \in \bbZ^n \mid  |x| = d - \delta \big \}$ 
$$
u : \ 
x 
\ \longmapsto \ 
\left \{
\begin{array}{ll}
x-e_i &  
\begin{array}{l}
\text{où $i$ est \textbf{le plus petit} indice tel que $x_i > 0$} \\
\text{si $x$ possède au moins deux composantes strictement positives}  
\end{array} \\ [0.4cm]
x & \text{sinon}
\end{array}
\right .
$$
dont les points fixes sont les $n$-uplets de $\bbZ^n_{d-\delta}$ n'ayant pas 
deux composantes strictement positives.
\noindent
Alors toute itération de $u$ conduit à un point fixe. 
\end{prop}

\begin{proof}
L'idée est de reprendre la preuve de~\ref{Chasseh=1} avec $\ell = n$ 
qui n'est jamais atteint en tant que \og plus petit indice $i$ tel que \fg{} 
en raison du \og au moins \textbf{deux} composantes \dots \fg{}.

Associons à $x \in \bbZ^n$ l'élément 
$y \in \bbZ^n_{0}$ défini par $y = x - s \varepsilon_n$ où $s = |x|$.
Si $x$ est fixé par $u$, c'est terminé, sinon le plus petit $i$ tel que $x_i > 0$ 
est nécessairement $<n$.
Par conséquent,
on a l'égalité $y_i = x_i$, a fortiori $x_i > 0 \Rightarrow y_i > 0$.
Or pour $y \in \bbZ^n_0$ et $y_i > 0$, on a $\calH_0(y -e_i) = \calH_0(y)-1 \in \bbN$.
Ainsi l'entier $\calH_0(y)$ diminue d'une unité après chaque itération de $u$.
Par conséquent, le processus finit par s'arrêter sur un point fixe de $u$.
\end{proof}

\subsubsection{$\bullet$ A l'aide d'une technique d'ordre monomial}

Le cas $\calM \subset \Jex_{2,d}$ avec $d \in \bbN$ est vraiment beaucoup plus facile 
que celui de $\calM \subset \Jex_{1,d}$ avec $d \leqslant \delta$ car
on peut le régler en utilisant un bon ordre monomial :

\begin{prop} \label{WkdQTriangSup}
Soit $\calM \subset \Jex_{2,d}$ et $\leqslant$
un ordre monomial sur $\bfA[\uX]$ tel que $X_1 > \cdots > X_n$.

\noindent
L'endomorphisme $W_{\calM}(\uQ)$ est triangulaire relativement à cet ordre monomial au sens où,
pour $X^\alpha$ fixé:
$$
W_{\calM}(\uQ)(X^\alpha) 
\ \in \ 
\bigoplus_{X^{{\alpha'}} \leqslant\, X^{\alpha}} \bfA X^{\alpha'}
$$

\noindent
Le déterminant $\det W_\calM(\uQ)$ est un produit de $p_i$ : 
$$
\det W_\calM(\uQ)
\ = \ 
\prod_{X^\alpha \in \calM} p_{\minDiv(X^\alpha)}
$$

\noindent
En particulier, pour $Q_i = X_i^{d_i} - X_i^{d_i-1}X_{i+1}$, on a $\det W_{\calM}(\uQ)  = 1$.
\end{prop}

\begin{proof}
Soit $X^\alpha \in \calM$ et $i = \minDiv(X^\alpha)$. 
On a 
$$
W_\calM(\uQ)(X^\alpha) \ = \ 
\pi_{\calM}\Big(\dfrac{X^\alpha}{X_i^{d_i}} Q_i\Big) \ = \ 
p_i \, X^\alpha \ + \ q_i \, \pi_{\calM}(X^{\alpha'}) 
\qquad
\text{avec $X^{\alpha'} = \dfrac{X^\alpha}{X_i} X_{i+1}$}
$$
Comparons $X^{\alpha'}$ et $X^\alpha$.  Puisque $i < n$ (car $\calM
\subset \Jex_2$), on a $X_{i+1} < X_i$, puis en multipliant par
$\frac{X^\alpha}{X_i}$, on obtient $X^{\alpha'} < X^\alpha$.

En ordonnant la base monomiale de $\calM$ de manière croissante pour $\leqslant$,
la matrice de $W_\calM$ est triangulaire supérieure 
à diagonale de $p_i$ ; d'où le résultat sur le déterminant.
\end{proof}

\index{rangement monomial triangulaire}%

\subsubsection{Exemple $D = (1,1,2)$ et $d=4$}

En prenant l'ordre lexicographique avec $X_1 > X_2 > X_3$ et en rangeant les
monômes de $\Jex_{2,d}$ de manière croissante, on a
$$
W_{2,d}(\uQ) \ = \ 
\EastBordermatrix{
p_{2} & q_{2} & q_{1} & . & . & . & . & . & . & . & \Heti{X_{2}X_{3}^{3}} \\ 
. & p_{2} & . & q_{1} & . & . & . & . & . & . & \Heti{X_{2}^{2}X_{3}^{2}} \\ 
. & . & p_{1} & . & . & . & . & . & . & . & \Heti{X_{1}X_{3}^{3}} \\ 
. & . & . & p_{1} & . & . & q_{1} & . & . & . & \Heti{X_{1}X_{2}X_{3}^{2}} \\ 
. & . & . & . & p_{1} & . & . & q_{1} & . & . & \Heti{X_{1}X_{2}^{2}X_{3}} \\ 
. & . & . & . & . & p_{1} & . & . & q_{1} & . & \Heti{X_{1}X_{2}^{3}} \\ 
. & . & . & . & . & . & p_{1} & . & . & . & \Heti{X_{1}^{2}X_{3}^{2}} \\ 
. & . & . & . & . & . & . & p_{1} & . & . & \Heti{X_{1}^{2}X_{2}X_{3}} \\ 
. & . & . & . & . & . & . & . & p_{1} & q_{1} & \Heti{X_{1}^{2}X_{2}^{2}} \\ 
. & . & . & . & . & . & . & . & . & p_{1} & \Heti{X_{1}^{3}X_{2}} \\ 
}
$$
En utilisant la technique de l'itérateur $w_{2,d}(\uR)$, on obtient le graphe suivant 
$$
\xymatrix{
&&&&X_1^2 X_3^2 \ar[d]^-{1} & \\
&& X_1^3 X_2 \ar[d]^-{1} & & X_1 X_2 X_3^2 \ar[d]^-{1} & \\
&&X_1^2 X_2^2 \ar[d]^-{1} & X_1^2 X_2 X_3 \ar[d]^-{1}  & X_2^2 X_3^2 \ar[d]^-{2} & X_1 X_3^3 \ar[dl]^-{1}  \\
&& X_1 X_2^3 \ar[d]^-{1} & X_1 X_2^2 X_3 \ar[d]^-{1} & X_2 X_3^3 \ar[d]^-{2} &\\
X_1^4 & X_1^3 X_3 & X_2^4 & X_2^3 X_3 & X_3^4 &
}
$$
où les puits sont exactement les monômes de degré $d=4$ n'appartenant pas à $\Jex_{2,d}$. 
Avec le rangement par strates, on obtient la matrice suivante :
$$
W_{2,d}(\uQ) \ = \
\EastBordermatrix{
p_{1} & . & . & q_{1} & . & . & . & . & . & . & \Heti{X_{1}X_{2}^{3}} \\ 
. & p_{1} & . & . & q_{1} & . & . & . & . & . & \Heti{X_{1}X_{2}^{2}X_{3}} \\ 
. & . & p_{2} & . & . & q_{1} & q_{2} & . & . & . & \Heti{X_{2}X_{3}^{3}} \\ 
. & . & . & p_{1} & . & . & . & q_{1} & . & . & \Heti{X_{1}^{2}X_{2}^{2}} \\ 
. & . & . & . & p_{1} & . & . & . & . & . & \Heti{X_{1}^{2}X_{2}X_{3}} \\ 
. & . & . & . & . & p_{1} & . & . & . & . & \Heti{X_{1}X_{3}^{3}} \\ 
. & . & . & . & . & . & p_{2} & . & q_{1} & . & \Heti{X_{2}^{2}X_{3}^{2}} \\ 
. & . & . & . & . & . & . & p_{1} & . & . & \Heti{X_{1}^{3}X_{2}} \\ 
. & . & . & . & . & . & . & . & p_{1} & q_{1} & \Heti{X_{1}X_{2}X_{3}^{2}} \\ 
. & . & . & . & . & . & . & . & . & p_{1} & \Heti{X_{1}^{2}X_{3}^{2}} \\ 
}
$$
On peut également effectuer le rangement par strates, mais au sein de chaque composante connexe :
$$
W_{2,d}(\uQ) \ = \
\EastBordermatrix{
p_{1} & q_{1} & . & \VR . & . & \VR . & . & . & . & . & \Heti{X_{1}X_{2}^{3}} \\ 
. & p_{1} & q_{1} & \VR . & . & \VR . & . & . & . & . & \Heti{X_{1}^{2}X_{2}^{2}} \\ 
. & . & p_{1} & \VR . & . & \VR . & . & . & . & . & \Heti{X_{1}^{3}X_{2}} \\ 
\HR{11} 
. & . & . & \VR p_{1} & q_{1} & \VR . & . & . & . & . & \Heti{X_{1}X_{2}^{2}X_{3}} \\ 
. & . & . & \VR . & p_{1} & \VR . & . & . & . & . & \Heti{X_{1}^{2}X_{2}X_{3}} \\ 
\HR{11} 
. & . & . & \VR . & . & \VR p_{2} & q_{1} & q_{2} & . & . & \Heti{X_{2}X_{3}^{3}} \\ 
. & . & . & \VR . & . & \VR . & p_{1} & . & . & . & \Heti{X_{1}X_{3}^{3}} \\ 
. & . & . & \VR . & . & \VR . & . & p_{2} & q_{1} & . & \Heti{X_{2}^{2}X_{3}^{2}} \\ 
. & . & . & \VR . & . & \VR . & . & . & p_{1} & q_{1} & \Heti{X_{1}X_{2}X_{3}^{2}} \\ 
. & . & . & \VR . & . & \VR . & . & . & . & p_{1} & \Heti{X_{1}^{2}X_{3}^{2}} \\ 
\noalign{\vskip-1pt}
}
$$

\label{NOTA06-w2diterateur}%
\index{strate}%

\subsection {Compléments sur les jeux \og simples\fg{} analogues au jeu circulaire}
\label{SectionJeuxSimples}

\centerline{\framebox [1.01\width][c]{Dans toute cette section,
nous écartons le cas $\delta=0$ i.e. le cas du format $D = (1,\dots,1)$}}
\medskip

En ce qui concerne le jeu circulaire, nous pouvons considérer que la plupart de ses
propriétés reposent sur l'existence de la fonction hauteur (cf.~\ref{FonctionHauteur})
dont une conséquence est le théorème du puits (cf.~\ref{TheoPuits}).
Le jeu circulaire possède-t-il des propriétés spécifiques? Qu'en est-il d'un jeu
de la forme $\uQ := \uX^D - \uR$ où $\uR = (R_1, \dots, R_n)$ est un jeu monomial
quelconque de format $D = (d_1, \dots, d_n)$?

\index{jeu!monomial!de format $D$}%

Il convient tout d'abord de revenir sur la notion de hauteur. Celle
pour le jeu circulaire, que nous avons notée $\calH_0
: \bbZ^n_0 \to \bbN$, possède des propriétés adaptées au système
générateur $(\varepsilon_i - \varepsilon_{i+1})_{1\leqslant i\leqslant n}$ de
$\bbZ^n_0$ où $(\varepsilon_1, \dots, \varepsilon_n)$ est la base
canonique de~$\bbZ^n$ et $\varepsilon_{n+1} = \varepsilon_1$. Nous l'avons
en particulier utilisée en degré~$\delta$, \idest{} sur $\bbZ^n_\delta \supset \bbN^n_\delta$,
par l'intermédiaire de la translation:
$$
\alpha = x + (\emouton), \qquad  x = \alpha - (\emouton) \qquad \qquad
\alpha \in \bbZ^n_\delta, \quad  x \in \bbZ^n_0
$$
Ici nous privilégions le degré $\delta$ et uniquement le degré $\delta$.
Via la correspondance $\alpha \leftrightarrow X^\alpha$, nous nous interrogeons sur l'existence 
d'une fonction hauteur~$h$, à valeurs dans $\bbN$, \emph{définie sur l'ensemble
des monômes de degré~$\delta$}, vérifiant
$$
h(X^\emouton) = 0, \qquad\qquad
h\Big(\dfrac{X^\alpha}{X_i^{d_i}} R_i\Big) = h(X^\alpha)-1
\quad \hbox {pour n'importe quel $i$ tel que }  X_i^{d_i} \mid X^\alpha
$$
L'objectif de cette section est d'une part de donner des
conditions nécessaires et suffisantes de l'existence de $h$
et de fournir un certain nombre de propriétés du
système $\uQ := \uX^D - \uR$ ou du système $\pXD + \bsq\,\uR \overset{\rm def.}{=}
(p_1X_1^{d_1} + q_1R_1, \cdots, p_nX_n^{d_n} + q_nR_n)$, en particulier
le lien avec la simplicité de $\uR$ (cf. la définition~\ref{DefJeuSimple}).

\index{hauteur!sur!1@les monômes de degré $\delta$}%

Avant d'aller plus loin, il est impératif d'une part de linéariser le problème en
attachant à $\uR$ une matrice carrée d'ordre $n$ à coefficients entiers
et d'autre part d'introduire une certaine combinatoire en associant à~$\uR$
un graphe orienté.

\subsubsection {Matrice $\Jac_\Un(\uX^D-\uR)$ attachée à $\uR=(R_1,\dots,R_n)$
et ses $n$ mineurs principaux $(r_1,\dots,r_n)$}

\index{matrice!attachée à un jeu monomial}%

Notons $\uQ := \uX^D-\uR$. La matrice $\Jac_\Un(\uQ)$ est la spécialisation
en $\uX := \Un$ de la matrice jacobienne de $\uQ$ avec la convention suivante
($\partial_j$ est l'opération de dérivée partielle $\frac{\partial}{\partial X_j}$):
$$
\begin{array}{c}
\\
\Jac_\uX(\uQ) =
\end{array}
\setlength{\arraycolsep}{0.5\arraycolsep}
\begin {array}{cc}
   &\partial_1\quad\cdots\quad\partial_n 
\\
\begin {array}{c} Q_1\\[1mm] \vdots\\[1mm] Q_n\end{array}
  &\left[\begin {array}{c}  
      \\[1mm]
      \quad\partial_j(Q_i) \\[1mm]
      \\[1mm]
  \end {array}\quad\right]
\\
\end {array}
$$
Cette matrice $\Jac_\Un(\uQ)$, que nous notons $A$ dans la suite, est
donc la matrice de $\bbM_n(\bbZ)$ dont les $n$ lignes sont les vecteurs
$\ell_1, \dots, \ell_n \in \bbZ^n_0$ définis de la manière suivante:
$$
\ell_i = \exposant(X_i^{d_i}) - \exposant(R_i)
$$
Précisément, 
en notant $(\varepsilon_i)_{1 \leqslant i \leqslant n}$ la base canonique de $\bbZ^n$,  
on a pour $R_i = X_1^{\alpha_{i_1}} X_2^{\alpha_{i_2}} \cdots X_n^{\alpha_{i_n}}$ :
$$
\ell_i = d_i \varepsilon_i - \sum_j \alpha_{ij} \varepsilon_j =
(-\alpha_{i1}, \dots, -\alpha_{i,i-1}, d_i-\alpha_{ii}, -\alpha_{i,i+1},\dots, -\alpha_{in})
$$
On notera $r_i$ le mineur principal de $A$ d'ordre $n-1$ obtenu en supprimant la ligne $i$
et colonne $i$. Nous verrons que c'est toujours un entier $\geqslant 0$.

\label{NOTA06-MatJac1Q}%
\index{hauteur!sur!2@$\bbN^n_\delta$}%

Une fonction hauteur peut alors être vue comme une fonction $h
: \bbN^n_\delta \to \bbN$ vérifiant la propriété:
$$
h(\emouton) = 0, \qquad\qquad
h(\alpha - \ell_i) = h(\alpha) - 1
\quad \hbox {pour n'importe quel $i$ tel que }  \alpha_i \geqslant d_i
$$
Nous insistons ici sur le fait que nous demandons à la fonction $h$ d'être
définie sur $\bbN^n_\delta$ et pas sur $\bbZ^n_\delta$ (ou sur son translaté
$\bbZ^n_0$, ce qui revient au même).

Cette matrice $A = (a_{ij})$ vérifie un certain nombre de propriétés susceptibles
d'être étudiées indépendamment les unes des autres.

$\blacktriangleright$
La somme de chaque ligne de $A$ est nulle \idest{} la somme des colonnes
de $A$ est nulle.

$\blacktriangleright$
Les coefficients non diagonaux sont $\leqslant 0$ donc ceux diagonaux sont $\geqslant 0$.

$\blacktriangleright$
Il s'ensuit que $a_{ii} = \sum_{j \ne i} |a_{ij}|$ pour tout $i$. A fortiori,
la matrice $A$ est faiblement dominante en lignes au sens suivant:
$a_{ii} \geqslant \sum_{j \ne i} |a_{ij}|$ pour tout $i$.
En conséquence, tout mineur principal de $A$ est $\geqslant 0$ comme le montre la
proposition suivante.

\index{matrice!faiblement dominante en lignes}%
\index{matrice!strictement dominante en lignes}%

\begin {prop} [Gershgorins's theorem and co] \label {Gershgorin} \leavevmode

Dans les trois premiers points, $A \in \bbM_n(\bbC)$ mais dans le dernier point $A$ est
réelle.

\begin {enumerate} [\rm i)]
\item
On suppose $\ker A$ non  réduit à $\{0\}$; alors il existe $1 \leqslant i \leqslant n$ tel que
$|a_{ii}| \leqslant \sum_{j \ne i} |a_{ij}|$.

\item
On suppose la diagonale de $A$ strictement dominante en lignes au sens où,
pour tout $i$ :
$$
|a_{ii}| > \sum_{j \ne i} |a_{ij}|
$$
Alors $A$ est inversible.

\item
Pour toute valeur propre $\lambda$ de $A$, il y a un $1 \leqslant i \leqslant n$ vérifiant
$|\lambda - a_{ii}| \leqslant \sum_{j \ne i} |a_{ij}|$.

\item
Soit $A \in \bbM_n(\bbR)$. On suppose que $A$ est faiblement dominante en lignes \idest{} pour tout $i$
$$
a_{ii}  \geqslant \sum_{j \ne i} |a_{ij}|  \qquad\qquad
\hbox {(en particulier $a_{ii} \geqslant 0$)}
$$
Alors tout mineur principal de $A$ est $\geqslant 0$.
\end {enumerate}
\end {prop}

\index{théorème!de Gershgorin}%
\index{matrice!strictement dominante en lignes}%
\index{mineur!principal (d'une matrice carrée)}%

\begin {proof} \leavevmode

i) Soit $x \in \bbC^n$ non nul vérifiant $Ax = 0$ et $i$ tel que $|x_i| = \max_j |x_j|$; on a donc
$$
a_{ii} = - \sum_{j \ne i} a_{ij} \frac {x_j}{x_i} \qquad \hbox {d'où} \qquad
|a_{ii}| \leqslant \sum_{j \ne i} |a_{ij}|
$$

ii) Par contraposée du point i).

iii)  Appliquer le point i) à $\lambda\Id_n - A$.

iv) Il suffit de montrer que $\det(A) \geqslant 0$. En effet, pour tout $I \subset \{1..n\}$,
la sous-matrice $A_{I \times I}$ vérifie les mêmes hypothèses de dominance en lignes que $A$
donc on aura $\det(A_{I \times I}) \geqslant 0$.

Utilisons le fait que $\det(A)$ est le produit des valeurs propres complexes de $A$.
Donc $\det(A) = \pi_1 \pi_2$ où $\pi_1$ est le produit
des valeurs propres non réelles et $\pi_2$ celui des valeurs propres réelles.
On a $\pi_1 > 0$ puisque pour toute valeur propre non réelle $\lambda$ de $A$,
$\overline\lambda$ est aussi valeur propre. Il suffit donc de montrer que toute
valeur propre réelle $\lambda$ de $A$ est $\geqslant 0$. Or, d'après le point iii), on a,
pour un certain $i$:
$$
|\lambda - a_{ii}| \leqslant \sum_{j \ne i} |a_{ij}|
$$
Et par hypothèse, on a $\sum_{j \ne i} |a_{ij}| \leqslant a_{ii}$; bilan: $|\lambda - a_{ii}| \leqslant a_{ii}$
d'où l'on déduit $\lambda \geqslant 0$.
\end {proof}

\begin {lem} \label{riliSumRelation}
Soit $A \in \bbM_n(\bfA)$ de lignes $\ell_1, \dots, \ell_n$ telle que
la somme des coefficients de chaque ligne $\ell_i$ soit nulle. On note
$r_i$ le mineur principal de $A$ d'ordre $n-1$ obtenu en supprimant la ligne $i$
et colonne $i$.

\medskip
\begin {enumerate}[\rm i)]
\item 
Soit $\widetilde A$ la matrice adjointe de $A$. Alors
$$
\widetilde A =
\begin {bmatrix}
r_1 &r_2 & \cdots &r_n \\
r_1 &r_2 & \cdots &r_n \\
\vdots &\vdots&  &\vdots  \\
r_1 &r_2 & \cdots &r_n \\
\end {bmatrix}
$$

\item
Les mineurs $r_i$ fournissent une relation linéaire entre les lignes $\ell_i$ ; précisément, $\sum_i r_i\ell_i = 0$.
\end {enumerate}
\end {lem}

\begin {proof} \leavevmode

i) Résulte du fait que la somme des coefficients de chaque ligne est nulle.
Montrons par exemple que $\widetilde {a}_{11} = \widetilde {a}_{21}$ i.e.
$\det(A^{1,1}) = -\det(A^{1,2})$ où $A^{i,j}$ désigne la matrice obtenue
en supprimant dans $A$ ligne~$i$ et colonne~$j$. Remarquons que $A^{1,1}$ et $A^{1,2}$ ont
des colonnes identiques à l'exception des premières; en utilisant le fait que la
somme des colonnes de $A$ est nulle:
$$
{\rm colonne}_1(A^{1,1}) = -{\rm colonne}_1(A^{1,2})  \qquad \hbox {d'où} \quad
\det(A^{1,1}) = -\det(A^{1,2})
$$

\medskip
ii) Résulte de $\widetilde A\,A = \det(A)\,\Id_n = 0$.
\end {proof}

\medskip
Dans la suite, nous allons prendre comme anneau de base $\bfA = \bbZ$ et nous interroger sur
les égalités suivantes dans lesquelles il est important de distinguer $\bbZ\,\ell_i$ et
$\bbN\,\ell_i$: 
$$
\bbZ^n_0 \overset{?}{=} \sum_i \bbZ\,\ell_i,
\qquad\qquad
\bbZ^n_0 \overset{?}{=} \sum_i \bbN\,\ell_i
$$
Bien entendu, si on a celle de droite, on a celle de gauche. Et réciproquement, si
on a celle de gauche et $-\ell_j \in \sum_i \bbN\,\ell_i$ pour chaque $j$, alors
on a celle de droite. Les mineurs principaux $r_1, \dots, r_n$ permettent une
caractérisation simple de ces égalités.

\begin {prop} \label {EasyLinearAlgebra}
On garde le contexte du lemme précédent en supposant l'anneau de base $\bfA=\bbZ$.

\begin {enumerate}[\rm i)]
\item
Les propriétés suivantes sont équivalentes:
\begin {enumerate}
\item
$\bbZ^n_0 = \bbZ\cdot\ell_1 + \cdots + \bbZ\cdot\ell_n$
\item
$\pgcd(r_1, \dots, r_n) = 1$.
\end {enumerate}
\noindent
Dans ce cas, le vecteur $r := \sum_i r_i\varepsilon_i$ est une base du module des relations
entre les lignes $\ell_i$ de $A$:
$$
\ker \transpose{A} = \bbZ\,.
\begin{bmatrix}
r_1 \\
\vdots \\
r_n 
\end{bmatrix}
$$

\item
On suppose de plus la matrice $A$ faiblement dominante en
lignes. Alors les $n$ mineurs principaux $r_i$ sont des entiers $\geqslant 0$ et on a l'équivalence:
\begin {enumerate}
\item
$\bbZ^n_0 = \bbN\cdot\ell_1 + \cdots + \bbN\cdot\ell_n$
\item
$\pgcd(r_1, \dots, r_n) = 1$ et $r_i \geqslant 1$ pour tout $i$.
\end {enumerate}
Dans ce cas, chaque $x \in \bbZ^n_0$ possède une unique \emph
{écriture positive minimale} $x = \sum_i m_i \ell_i$,
positive signifiant $m_i \in \bbN$ pour tout $i$, et minimale qu'elle
minimise la somme~$\sum_i m_i$.
\end {enumerate}
\end {prop}

\index{e@écriture positive minimale $\sum_i m_i\ell_i$}

\begin {proof} \leavevmode

i) Considérons $\transpose{A} : \bbZ^n \to \bbZ^n$ d'image contenue
dans $\bbZ^n_0$, que l'on voit comme une application linéaire
$U : \bbZ^n \to \bbZ^n_0 \simeq \bbZ^{n-1}$. L'écriture de $u \in \bbZ^n$
appartenant à $\bbZ^n_0$ sur la base $(\varepsilon_1-\varepsilon_n,
\varepsilon_2-\varepsilon_n, \dots, \varepsilon_{n-1}-\varepsilon_n)$ de
$\bbZ^n_0$ est
$$
u_1\varepsilon_1 + \cdots + u_n\varepsilon =
u_1(\varepsilon_1 - \varepsilon_n) + u_2(\varepsilon_2 - \varepsilon_n) + \cdots +
u_{n-1}(\varepsilon_{n-1} - \varepsilon_n) 
$$
Ainsi les $n-1$ lignes de $U$ sont les $n-1$ premières lignes
de $\transpose{A}$. La matrice~$U$, à $n$ colonnes, $n-1$ lignes, est donc la transposée
de $A_{\{1..n\} \times \{1..n-1\}}$, matrice extraite de $A$ sur les $n-1$ premières colonnes. 

On en déduit que les $(r_1, \dots, r_n)$ sont, au signe près, les $n$ mineurs d'ordre
$n-1$ de $U \in \bbM_{n-1,n}(\bbZ)$. Il s'ensuit que $\Im\transpose{A}$ est d'indice fini
dans $\bbZ^n_0$ si et seulement si un des $r_i$ est non nul, auquel cas :
$$
\#\bigl(\bbZ^n_0/\Im\transpose{A}\bigr) = |\pgcd(r_1, \dots, r_n)|
$$
D'où l'équivalence entre les propriétés {\it(a)} et {\it(b)}.

\smallskip
Pour démontrer que $r := \sum_i r_i\varepsilon_i$ est une base de
$\ker\transpose{A}$, utilisons le fait que $\transpose{A}
: \bbZ^n \twoheadrightarrow
\bbZ^n_0$ est une surjection. Ceci prouve que le sous-module $\ker\transpose{A}$ de $\bbZ^n$
est libre de rang~1 et facteur direct dans~$\bbZ^n$, donc engendré par un vecteur dont
les composantes sont premières entre elles. Mais comme 
$r$ a ses composantes premières entre elles et appartient au noyau
$\ker\transpose{A}$, le vecteur $r$ est une base de $\ker\transpose{A}$.

\medskip
ii) La positivité des $r_i$ résulte de la proposition \ref{Gershgorin}.

\smallskip
Supposons (a), a fortiori $\bbZ^n_0 = \bbZ\cdot\ell_1 + \cdots
+ \bbZ\cdot\ell_n$ donc $\pgcd(r_1, \dots, r_n) = 1$.
Puisque $-\sum_i \ell_i \in \bbZ^n_0$, on peut écrire
$-\sum_i \ell_i = \sum_i m_i \ell_i$ avec $m_i \in \bbN$. Il vient
$\sum_i (m_i+1)\ell_i = 0$ d'où un $q \in \bbZ$ tel que $m_i + 1 =
qr_i$ pour tout $i$. On en déduit que $r_i \ne 0$ pour tout $i$
donc $r_i \geqslant 1$.

Supposons (b).
On sait alors d'après i) que $\bbZ^n_0 = \bbZ\cdot\ell_1 + \cdots
+ \bbZ\cdot\ell_n$ et il s'agit de remplacer $\sum_i \bbZ\,\ell_i$ par
$\sum_i \bbN\,\ell_i$.  Pour cela, on écrit $x \in \bbZ^n_0$ comme une
combinaison entière quelconque de $\ell_1, \dots, \ell_n$ sans se
soucier de positivité et on utilise $\sum_i r_i\ell_i = 0$:
$$
x = \sum_i a_i \ell_i = \sum_i (a_i - mr_i) \ell_i  \qquad
a_i \in \bbZ, \quad \forall m \in \bbZ
$$
Puis on choisit $m \in \bbZ$ tel que pour tout $i$, on ait $a_i -
mr_i \geqslant 0$ i.e. $m \leqslant a_i/r_i$, ce qui est possible car $r_i \geqslant 1$. En prenant
$$
m = \min_{i=1,\dots, n} \left\lfloor \frac {a_i}{r_i} \right\rfloor
$$
on obtient une écriture $x = \sum_i m_i\ell_i$ sur $\sum_i \bbN\,\ell_i$ (écriture
minimale au sens minimalité de~$\sum_i m_i$).

\smallskip
Pour l'unicité de l'écriture positive minimale, il suffit, étant données deux égalités
$$
\sum_i m_i\ell_i = \sum_i m'_i \ell_i  \qquad \hbox {avec} \quad m_i, m'_i \in \bbZ 
$$
de montrer l'implication
$$
\sum_i m_i = \sum_i m'_i   \ \Longrightarrow\   m = m'
$$
On a $\sum_i (m_i - m'_i)\ell_i = 0$, d'où, d'après i), l'existence 
de $q \in \bbZ$ tel que $m - m' = q\,r$. Il vient alors
$0 = \sum_i m_i - \sum m'_i = q \sum_i r_i$ et puisque $\sum_i r_i$ n'est pas nul,
$q=0$ i.e. $m = m'$.
\end {proof}

\subsubsection {Graphe orienté étiqueté associé à un jeu monomial $\uR$} 
\label {GraphOfR}

\index{graphe (orienté étiqueté) d'un jeu monomial}%

Il s'agit du graphe orienté dont les sommets sont les monômes de degré $\delta$, avec des arcs étiquetés:
un arc $X^\alpha \overset{i}{\to} X^\beta$ pour chaque $i$ vérifiant
$$
\hbox {$X_i^{d_i} \mid X^\alpha$ et $X^\beta = \frac {X^\alpha}{X_i^{d_i}} R_i$}
\leqno (\hbox {règle de réécriture numéro $i$})
$$
Ce graphe ne doit pas être confondu avec le graphe d'un itérateur $w_\calM$ qui
est piloté par $\minDiv$. Le graphe dont nous parlons ici est indépendant de
tout mécanisme de sélection.

\medskip

On a d'une part $X^\alpha - X^\beta \in \langle\uX^D-\uR\rangle$ puisque
$$
X^\alpha - X^\beta = \frac{X^\alpha}{X_i^{d_i}} (X_i^{d_i} - R_i) \in \langle\uX^D-\uR\rangle
$$
Et d'autre part:
$$
\alpha - \beta = \exposant(X_i^{d_i}) - \exposant(R_i) \overset {\rm def}{=} \ell_i
$$
A noter que le nombre d'arcs sortant de $X^\alpha$ est le cardinal de
l'ensemble de divisibilité de $X^\alpha$; en particulier,
aucun arc ne sort du mouton-noir.
Dans la suite, on supposera $R_i \ne X_i^{d_i}$, ce qui équivaut au fait que le graphe de $\uR$
ne possède pas de boucle.

\medskip

En voici deux exemples. Le premier pour le format $D = (1,1,2)$ de degré critique $\delta = 1$. En prenant
$\uR = (X_2, X_1, X_2^2)$, on obtient le graphe ci-dessous. A droite, on a fait figurer
la matrice $A=\Jac_\Un(\uX^D-\uR)$ attachée à $\uR$ ainsi que ses 3 mineurs principaux.
$$
\xymatrix {
\vertex{X_1} \ar@/^7pt/[r]^1 &\vertex{X_2}\ar@/^7pt/[l]^2 & \vertex{X_3}
}
\qquad\qquad
A = \begin {bmatrix}
1 & -1 & 0 \\
-1 & 1 & 0 \\
0  & -2 & 2 \\
\end {bmatrix}
\qquad
(r_1, r_2, r_3) = (2, 2, 0)
$$
Dans le second, $D = (3,2,1)$ de degré critique $\delta = 3$, de
mouton-noir $X_1^2X_2$, et $\uR = (X_2^2X_3, X_1X_2,X_1)$.
$$
\let\V=\vertex
\xymatrix {
                  &                    &                     &                   &\V{X_2X_3^2}\ar[r]^3 &
                         \V{X_1X_2X_3}\ar[rd]^3  \\
\V{X_3^3}\ar[r]^-{3} &\V{X_1X_3^2}\ar[r]^3 &\V{X_1^2X_3}\ar[r]^3 &\V{X_1^3}\ar[r]^1 &\V{X_2^2X_3}\ar[ru]^2\ar[rd]_3 &&
                          \V{X_1^2X_2} \\
                  &                    &                     &                   &\V{X_2^3}\ar[r]_2 &
                         \V{X_1X_2^2}\ar[ru]_2  \\
}
$$
Voici la matrice $A=\Jac_\Un(\uX^D-\uR)$ attachée à $\uR$ et ses mineurs principaux:
$$
A = \begin {bmatrix}
3 & -2 & -1 \\
-1 & 1 & 0 \\
-1  & 0 & 1\\
\end {bmatrix}
\qquad
(r_1, r_2, r_3) = (1, 2, 1)
$$
Le lemme suivant, immédiat, est à la charge du lecteur.

\index{matrice!attachée à un jeu monomial}%

\begin {lem}

Soit un graphe fini orienté muni d'une fonction $h$ sur les sommets, à valeurs dans $\bbN$, vérifiant
$h(\beta) = h(\alpha)-1$ pour tout arc $\alpha \to \beta$. Pour un sommet $\alpha$, on considère
le processus suivant: prendre, s'il en existe, un arc quelconque $\alpha \to \beta$ (i.e. sortant de $\alpha$)
et itérer cette construction avec $\beta$.
\begin {enumerate} [\rm i)]
\item
Quelque soit le sommet $\alpha$, le processus se termine sur un sommet $\beta$ sans arc sortant.

\item
Le graphe est sans circuit.

\item
Etant donnés deux chemins $\xymatrix {\alpha\ar@/^8pt/@{~>}[rr]\ar@/_8pt/@{~>}[rr] & &\beta}$, ils
ont même longueur $h(\alpha) - h(\beta)$.
\end {enumerate}
\end {lem}

\medskip
Un des résultats importants et plus délicat est le suivant.

\begin {lem} \label{LemmeExistenceHauteur}
Soit $\uR$ un jeu monomial vérifiant:
$$
\forall\, \alpha \in \bbN^n_\delta, \quad \alpha - (\emouton) \in \sum_i \bbN\,\ell_i 
$$
Alors, on dispose des propriétés suivantes
\begin{enumerate} [\rm i)]
\item
On a l'égalité $\bbZ^n_0 = \sum_i\bbZ\,\ell_i$.

\item
Pour chaque $\alpha \in \bbN^n_\delta$, il y a une unique \emph
{écriture positive minimale} $\alpha -(\emouton) = \sum_i m_i \ell_i$,
positive signifiant $m_i \in \bbN$ pour tout $i$, et minimale qu'elle
minimise la somme~$\sum_i m_i$.

\item
Il existe une fonction hauteur $h$ pour $\uR$.
\end {enumerate}
\end {lem}

\index{hauteur!sur!2@$\bbN^n_\delta$}%
\index{e@écriture positive minimale $\sum_i m_i\ell_i$}%

\begin {proof} \leavevmode

i) Pour $\alpha, \beta \in \bbN^n_\delta$, on déduit de l'hypothèse que
$\alpha - \beta \in \sum_i \bbZ \ell_i$.  Pour obtenir $\bbZ^n_0
= \sum_i\bbZ\,\ell_i$, il suffit de voir que tout vecteur de
$\bbZ^n_0$ est la différence de deux vecteurs de $\bbN^n_\delta$ (ici
on va utiliser $\delta \geqslant 1$), ce qui revient à le montrer pour les
$\varepsilon_j - \varepsilon_k$ où $(\varepsilon_i)_{1\le i\le n}$ est la base
canonique de $\bbZ^n$. Par exemple:
$$
\varepsilon_1 - \varepsilon_2 \overset {\rm def}{=} (1, -1, 0, \dots, 0) =
(1, \delta-1, 0, \dots, 0)  - (0, \delta, 0, \dots, 0)  
$$
L'égalité $\bbZ^n_0 = \sum_i\bbZ\,\ell_i$ nous permet d'utiliser
(proposition~\ref{EasyLinearAlgebra}) le fait que le vecteur $r = \sum_i r_i \varepsilon_i$
est une base du module des relations entre les vecteurs $\ell_i$ et
l'égalité $\pgcd(r_1, \dots, r_n) = 1$.

\medskip
ii) Pour la preuve de l'unicité, reprendre celle du point ii) de la proposition~\ref{EasyLinearAlgebra}.

\medskip
iii)
On définit la fonction $h : \bbN^n_\delta \to \bbN$ de la manière suivante:
$$
\alpha \mapsto \sum_i m_i \quad \hbox {pour l'écriture positive minimale} \quad
\alpha - (\emouton) = \sum_i m_i\ell_i, \ m_i \in \bbN
$$
Il reste à prouver que $h$ est une fonction hauteur.
Elle vérifie clairement $h(\emouton) = 0$.
Nous allons montrer que $\alpha_j \geqslant d_j$ entraine $m_j \geqslant 1$; il s'en suivra que l'écriture positive
minimale de $\alpha - \ell_j$ est
$$
\alpha - \ell_j = \sum_{i < j} m_i\ell_i + (m_j-1) \ell_j + \sum_{k > j} m_k \ell_k
$$
donc que $h(\alpha - \ell_j) = h(\alpha)-1$.

Soit $A = (a_{i,j})_{1\le i,j\le n}$ la matrice de lignes $\ell_1, \dots, \ell_n$. La composante sur $\varepsilon_j$
de l'égalité $\alpha - (\emouton) = \sum_i m_i\ell_i$ fournit:
$$
\alpha_j - (d_j- 1) = m_j a_{jj} + \sum_{i \ne j} m_i a_{ij}
$$
Mais on a $m_i a_{ij} \leqslant 0$ car les $m_i$ sont $\geqslant 0$ et les
$a_{ij} \leqslant 0$ pour $i \ne j$. Il vient donc $m_j a_{jj} \geqslant \alpha_j
- (d_j- 1)$ et comme on a supposé $\alpha_j \geqslant d_j$, on a $\alpha_j - (d_j- 1) \geqslant 1$
dont on déduit $m_j \ne 0$ \idest{} $m_j \geqslant 1$.
\end {proof}

\begin {theo}
\label{EquivalencesExistenceHauteur}
Etant donné un jeu monomial $\uR$, les propriétés suivantes sont équivalentes.

\begin {enumerate} [\rm i)]
\item
Le graphe de $\uR$ est sans circuit.

\item
Pour chaque sommet $X^\alpha$ du graphe de $\uR$, il y a un chemin reliant
$X^\alpha$ à $X^\emouton$.

\item
Pour tout $\alpha \in \bbN^n_\delta$, on a $\alpha - (\emouton) \in \sum_i \bbN\,\ell_i$.

\item
Il existe une fonction hauteur $h$ pour $\uR$.
\end {enumerate}
\end {theo}

\index{graphe (orienté étiqueté) d'un jeu monomial}%

\begin {proof} \leavevmode
On rappelle que l'on a supposé $R_i \ne X_i^{d_i}$ de manière à ce que le graphe de $\uR$
ne possède pas de boucle.

\medskip
\noindent
i) $\Rightarrow$ ii)
Dans un graphe fini sans circuits, pour tout sommet, il existe un chemin de ce sommet
à un sommet sans arc sortant. Ici, il y a un seul sommet sans arc sortant: c'est
le \MoutonNoir $X^\emouton$.

\medskip
\noindent
ii) $\Rightarrow$ iii)
Pour un arc $X^\alpha \overset{i}{\to} X^\beta$ du graphe de $\uR$, on a $\alpha - \beta = \ell_i$
donc pour un chemin joignant $X^\alpha$ à $X^\beta$, on a $\alpha - \beta \in \sum_i \bbN\ell_i$.
D'où le résultat puisque par hypothèse, il y a un chemin reliant $X^\alpha$ à $X^\emouton$.

\medskip
\noindent
iii) $\Rightarrow$ iv) C'est l'objet du lemme précédent.

\medskip
\noindent
iv) $\Rightarrow$ i)
Découle du fait qu'en présence d'un arc $X^\alpha \to X^\beta$, on a $h(X^\beta) = h(X^\alpha) - 1$.
\end {proof}

\medskip

La construction de la fonction hauteur dans le lemme~\ref{LemmeExistenceHauteur}
et la preuve du théorème fournissent en prime le corollaire suivant.

\begin {coro} \label {CorollaireExistenceHauteur}
Soit $\uR$ un jeu monomial vérifiant les conditions équivalentes
du théorème précédent.
On note $\ell_1, \dots, \ell_n$ les lignes de la matrice associée et $r_1, \dots, r_n$ les $n$
mineurs principaux.

\begin {enumerate} [\rm i)]
\item
On a $\pgcd(r_1, \dots, r_n) = 1$.

\item
La fonction hauteur $h$ est unique.

\item
Tous les chemins reliant $X^\alpha$ à $X^\beta$ ont même longueur $h(X^\alpha) - h(X^\beta)$.
En particulier, tous les chemins reliant $X^\alpha$ à $X^\emouton$ ont même longeur $h(X^\alpha)$.

\item
Pour n'importe quel chemin reliant $X^\alpha$ à $X^\beta$, en notant $m_i$ le nombre
d'étiquettes~$i$ de ce chemin:
$$
\alpha - \beta = \sum_i m_i \ell_i   \qquad h(X^\alpha) - h(X^\beta) = \sum_i m_i
$$
En particulier, pour n'importe quel chemin reliant $X^\alpha$ à $X^\emouton$, en notant $m_i$ le nombre
d'étiquettes~$i$ de ce chemin:
$$
\alpha - (\emouton) = \sum_i m_i \ell_i   \qquad h(X^\alpha) = \sum_i m_i
$$

\item
Pour $\alpha \in \bbN^n_\delta$, il y a une écriture positive minimale (minimale au sens $\sum_i m_i)$
qui est unique:
$$
\alpha - (\emouton) = \sum_i m_i \ell_i  \qquad m_i \in \bbN
$$
Et elle est réalisée par tout chemin reliant $X^\alpha$ à $X^\emouton$.
\end {enumerate}
\end {coro}

\bigskip

Voici, avec $n=3$, l'exemple d'un jeu $\uR$ ne possédant pas de fonction hauteur
mais vérifiant $\Im\transpose{A} = \bbZ^3_0$ ou encore de manière
équivalente $\pgcd(r_1, r_2, r_3) = 1$.
Considérons $\uR = (X_3, X_1X_2, X_1X_3)$ de format
$(1,2,2)$, de degré critique $\delta = 2$, de mouton-noir $X_2X_3$. La matrice
$A=\Jac_\Un(\uX^D-\uR)$ attachée à $\uR$ est
$$
A = \begin {bmatrix}
 1 & 0 &-1 \\
-1 & 1 & 0 \\
-1 & 0 & 1 \\
\end {bmatrix}
\qquad
(r_1, r_2, r_3) = (1, 0, 1)
$$
On voit que les lignes $\ell_1, \ell_2$ forment une base de $\bbZ^3_0$
et on a la relation $\ell_1 + \ell_3 = 0$. Cependant, on a un circuit
$\xymatrix {\vertex{X_1X_3} \ar@/^7pt/[r]^1 &\vertex{X_3^2}\ar@/^7pt/[l]^3}$
et 2 composantes connexes:
$$
\xymatrix {
\vertex{X_1^2} \ar[r]^1 &\vertex{X_1X_3} \ar@/^7pt/[r]^1 &\vertex{X_3^2}\ar@/^7pt/[l]^3
}
\qquad\qquad\qquad
\xymatrix {
\vertex{X_2^2} \ar[r]^2 &\vertex{X_1X_2} \ar[r]^1 &\vertex{X_2X_3}
} 
$$
Le système $\uR$ n'est pas simple au sens de la définition~\ref{DefJeuSimple}.
En effet, en notant $\uQ = \uX^D-\uR = (X_1-X_3, X_2^2 - X_1X_2, X_3^2 - X_1X_3)$,
ce dernier système, en plus du point unité $(1:1:1)$, admet $(1:0:1)$ comme zéro.
En notant l'idéal de chacun de ces 2 points projectifs:
$$
I_{(1:1:1)} = \langle X_1-X_2, X_1-X_3\rangle, \qquad
I_{(1:0:1)} = \langle X_1-X_3, X_2\rangle
$$
le lecteur pourra vérifier
que
$$
\langle\uQ\rangle^\sat = I_{(1:1:1)} \cap I_{(1:0:1)} =
\langle X_1-X_3,\ X_2(X_2-X_3)\rangle =
\langle X_1-X_3,\ X_2(X_1-X_2)\rangle
$$
En ce qui concerne le lien entre la non-existence d'une fonction hauteur et la non-simplicité de $\uR$,
se référer à la conjecture~\ref{ConjectureJeuSimple}.

\index{matrice!attachée à un jeu monomial}%
\index{jeu!monomial!simple}%

\medskip
L'exemple ci-dessous montre que l'existence d'une fonction hauteur ne peut pas se déduire du
seul examen de la matrice attachée à $\uR$. Il s'agit de $\uR = (X_1X_3, X_1, X_1X_3)$, de format
$D = (2,1,2)$, de degré critique $\delta=2$, de mouton-noir $X_1X_3$. Sa matrice $\Jac_\Un(\uX^D-\uR)$
est la même que celle de l'exemple précédent et pourtant il admet une
fonction hauteur vu que son graphe est sans circuit:
$$
\xymatrix @R=0.4cm{
                    &                    &                    &\vertex{X_2X_3}\ar[d]^2 \\
\vertex{X_2^2}\ar[r]^2 &\vertex{X_1X_2}\ar[r]^2 &\vertex{X_1^2}\ar[r]^1 &\vertex{X_1X_3}  \\
                   &                     &                    &\vertex{X_3^2}\ar[u]_3 \\ 
}
$$

\index{graphe (orienté étiqueté) d'un jeu monomial}%

\bigskip

Dans la proposition suivante, nous abordons, pour un jeu monomial $\uR$,
l'existence d'une fonction hauteur $h$ \emph {définie sur $\bbZ^n_\delta$ tout entier}
\idest{}  $h : \bbZ^n_\delta \to \bbN$
vérifiant la propriété:
$$
h(\emouton) = 0, \qquad\qquad
h(\alpha - \ell_i) = h(\alpha) - 1
\quad \hbox {pour n'importe quel $i$ tel que }  \alpha_i \geqslant d_i
$$
Ou encore, de manière équivalente par translation, l'existence d'une
fonction hauteur $h_0 : \bbZ^n_0 \to \bbN$ vérifiant $h_0({\bf 0}) =
0$ et $h_0(x - \ell_i) = h_0(x)-1$ pour tout $i$ tel que $x_i \geqslant 1$.

\index{hauteur!sur!3@$\bbZ^n_0$}%

\begin {prop} \label{ExistenceSuperHauteur}

Pour un jeu monomial $\uR$ dont on note $\ell_1, \dots, \ell_n$ les
lignes de sa matrice attachée $\Jac_\Un(\uX^D-\uR)$  et $r_1, \dots, r_n$ les $n$
mineurs principaux, on a l'équivalence:

\begin {enumerate}[\rm i)]
\item
$\pgcd(r_1, \dots, r_n) = 1$ et $r_i \geqslant 1$ pour tout $i$.

\item
Il existe une fonction hauteur $h_0 : \bbZ^n_0 \to \bbN$.
\end {enumerate}

\end {prop}

\begin {proof} \leavevmode

\noindent
i) $\Rightarrow$ ii)
D'après le point ii) de la proposition \ref{EasyLinearAlgebra}, tout $x \in \bbZ^n_0$
possède une unique écriture positive minimale $x = \sum_i m_i\,\ell_i$ avec $m_i \in \bbN$.
On pose alors
$$
h_0(x) = \sum_i m_i
$$
On vérifie, de manière analogue au point iii) du lemme~\ref{LemmeExistenceHauteur},
que $h_0$ est une fonction hauteur.

\smallskip
\noindent
ii) $\Rightarrow$ i)
D'après le point ii) de la proposition \ref{EasyLinearAlgebra}, il
suffit de montrer que $\bbZ^n_0 = \sum_i \bbN\, \ell_i$. Soit
$x \in \bbZ^n_0$. Pour prouver l'appartenance
$x \in \sum_i \bbN\,\ell_i$, raisonnons par récurrence sur
$h_0(x)$. Si $h_0(x) = 0$, alors $x$ est nul et l'appartenance est immédiate.
Sinon, $x$ étant non nul et dans $\bbZ^n_0$, il
existe un indice $j$ tel que $x_j > 0$, donc $h_0(x - \ell_j) =
h_0(x) -1$ et par récurrence, $x-\ell_j \in \sum_i \bbN\,\ell_i$;
il en est alors de même de $x = (x-\ell_j) + \ell_j$.
\end {proof}

\subsubsection {Quelques exemples}

\medskip
$\blacktriangleright$
Pour le jeu circulaire $\uR = (X_1^{d_1-1}X_2, X_2^{d_2-1}X_3,\dots,
X_n^{d_n-1}X_1)$, on a $\ell_i = \varepsilon_i - \varepsilon_{i+1}$ où
$(\varepsilon_i)_{1 \leqslant i \leqslant n}$ est la base canonique de
$\bbZ^n$. On vérifie aisément que tous les $r_i$ sont égaux à 1. En
utilisant~\ref{ExistenceSuperHauteur}, on retrouve l'existence d'une
hauteur $h_0$ définie sur $\bbZ^n_0$.

\medskip
$\blacktriangleright$
Soit le jeu $\uR^{(1)} = (X_1^{d_1-1}X_2, X_1^{d_2-1}X_3, \dots, X_1^{d_n-1}X_1)$,
\idest{} via une indexation cyclique modulo~$n$:
$$
R^{(1)}_i = X_1^{d_i-1} X_{i+1}  \qquad \qquad (X_{n+1} = X_1)
$$
Bien qu'analogue au jeu circulaire, il privilégie~$X_1$. 
On a $\ell_i = d_i\varepsilon_i - (d_i-1)\varepsilon_1 - \varepsilon_{i+1}$,
ce qui donne comme  matrice attachée $A = \Jac_\Un\big(\uX^D - \uR^{(1)}\big)$
$$
A = \begin {bmatrix}
1        & -1  & 0     & \cdots   & 0  \\
1-d_2    & d_2 & -1     & \cdots  & 0 \\
\vdots   &     &\ddots &          &\vdots\\
1-d_{n-1} & 0   &\cdots &d_{n-1}   & -1 \\
-d_n     & 0   &\cdots &  0       & d_n \\
\end {bmatrix}
$$
Le caractère échelonné des $n-1$ premières lignes $\ell_1, \dots, \ell_{n-1}$
et la présence des $-1$ à droite montrent que celles-ci forment une base de
$\bbZ^n_0$. Le point (iii) de la proposition~\ref{EasyLinearAlgebra}
fournit alors le fait que $\pgcd(r_1, \dots, r_n) = 1$ et que $\sum_i
r_i \varepsilon_i$ est une base du module des relations entre les
$\ell_i$.

\index{jeu!monomial!de format $D$@$R^{(1)}$}%

Par ailleurs, il est facile de vérifier que l'on $\sum_i r'_i\ell_i = 0$ en prenant:
$$
r'_1 = d_2 \cdots d_n, \qquad r'_2 = d_3 \cdots d_n,\qquad \dots,
\qquad r'_{n-1} = d_n,\qquad r'_n = 1
$$
Il existe donc $k \in \bbZ$ tel que $r'_i = kr_i$ pour tout $i$.
Et comme $r_1 = d_2\cdots d_n = r'_1$, c'est que $k = 1$.
Bilan: les $r'_i$ sont les $n$ mineurs principaux de la matrice attachée~$A$.

En utilisant~\ref{ExistenceSuperHauteur}, on obtient l'existence d'une
hauteur $h_0$ définie sur $\bbZ^n_0$.

\label{NOTA06-jeuR1}%

\medskip
$\blacktriangleright$
Pour illustrer la proposition~\ref{ExistenceSuperHauteur}, 
voici l'exemple d'un jeu possédant une fonction hauteur
sur~$\bbN^n_\delta$ mais pas sur $\bbZ^n_\delta$.  On prend $n = 3$,
$D = (1,1,2)$ donc $\delta = 1$ et
$$
\uR = (X_3, X_1, X_1X_3) 
$$
Le mouton-noir est~$X_3$ et le graphe de $\uR$ est sans circuit
$$
\xymatrix {
X_2 \ar[r] & X_1\ar[r] & X_3
}
$$
d'où l'existence d'une fonction hauteur (sous-entendu sur
$\bbN^n_\delta \overset{\rm ici}{=} \bbN^3_1$).
Voici la matrice $A = \Jac_\Un(\uX^D-\uR)$ attachée à~$\uR$:
$$
A = \EastBordermatrix{
1  &0 & -1  &\ell_1 \\
-1 &1  &0   &\ell_2 \\
-1 &0  &1   &\ell_3 \\
},
\qquad\qquad
\ell_1 + \ell_3 = 0
\qquad
r_1 = 1, \quad r_2 = 0, \quad r_3 = 1
$$
Donc $(\ell_1, \ell_2)$ est une base de $\bbZ^3_0$, a fortiori
$\bbZ^3_0 = \bbZ\,\ell_1 + \bbZ\,\ell_2 + \bbZ\,\ell_3$
en accord avec $\pgcd(r_1, r_2, r_3) = 1$. Mais comme $r_2 = 0$, on n'a pas
$\bbZ^3_0 = \bbN\,\ell_1 + \bbN\,\ell_2 + \bbN\,\ell_3$.
On laisse le soin au lecteur de vérifier que 
$-\ell_2 \notin \bbN\,\ell_1 + \bbN\,\ell_2 + \bbN\,\ell_3$.

\smallskip
On peut également se convaincre qu'il n'y a pas de fonction hauteur $h_0$ sur
$\bbZ^3_0$, en prenant comme vecteur $x = -\ell_2 = (1,-1,0)$. Il y a un seul
$i$ tel que $x_i \geqslant 1$, à savoir $i=1$. On réalise $x' = x - \ell_1 =
(0, -1, 1)$; de nouveau, un seul $i$ tel que $x'_i \geqslant 1$, à savoir $i = 3$. On
réalise $x' - \ell_3$, et on retombe sur $x$ (puisque $x - \ell_1
- \ell_3 = x$ vu que $\ell_1 + \ell_3 = 0$).
Et on est donc en présence d'une boucle infinie $\xymatrix {x\ar@/^7pt/[r] &x'\ar@/^7pt/[l]}$.

\index{graphe (orienté étiqueté) d'un jeu monomial}%

\subsubsection {Pour un système monomial $\uR$, quels sont les bénéfices d'une fonction hauteur?}

Sans entrer tout de suite dans les détails, on peut affirmer que les
résultats \framebox [1.1\width][c]{en degré $\delta$} montrés pour le
jeu circulaire restent valides pour tout jeu monomial $\uR$ admettant
une fonction hauteur. Par exemple, en notant $\uQ = \uX^D - \uR$, on
peut ranger la base monomiale de $\Jex_{1,\delta}$ de sorte que la
matrice $W_{1,\delta}(\uQ)$ soit triangulaire (à diagonale unité). En
conséquence, $\omega_\uQ(X^\emouton) = 1$ puis $\omega_\uQ(X^\alpha) =
1$ pour tout $X^\alpha$ de degré $\delta$, etc. Nous allons évoquer
quelques-uns de ces résultats en commençant par un qui ne privilégie
pas $\minDiv$.

\smallskip

\begin {prop} \label {HauteurSimplicite} Soit $\uR$ un jeu monomial admettant une fonction hauteur
et $\uQ = \uX^D - \uR$.

\begin {enumerate} [\rm i)]
\item
On a $X^\alpha - X^\beta \in \langle\uQ\rangle$ pour $|\alpha| = |\beta|
= \delta$.

\item 
Le système $\uQ$ admet comme  seul zéro projectif le point $\Un = (1 : \cdots : 1)$
et celui-ci est simple. Ou encore, de manière équivalente,
le saturé $\langle\uQ\rangle^\sat$ est l'idéal $I = \langle X_i-X_j, 1 \leqslant 
i,j \leqslant n\rangle$ de ce point.
\end {enumerate}
\end {prop}

\begin {proof} \leavevmode

i) En présence d'un arc $X^\alpha \to X^\beta$, on a
$X^\alpha - X^\beta \in \langle\uQ\rangle$. Redonnons l'argument vu
lors de l'introduction (page \pageref{GraphOfR}) du graphe de $\uR$: il y a un $i$ tel que
$X_i^{d_i} \mid X^\alpha$ et $X^\beta = R_i X^\gamma$  où $X^\gamma$ est défini par
$X^\alpha = X_i^{d_i} X^\gamma$, auquel cas
$$
X^\alpha - X^\beta = (X_i^{d_i} - R_i)X^\gamma \overset{\rm def.}{=} Q_i X^\gamma \in \langle\uQ\rangle
$$
Par itération, on en déduit qu'il en
est de même en présence d'un chemin reliant $X^\alpha$ à $X^\beta$. Par conséquent,
on a $X^\alpha - X^\emouton \in \langle\uQ\rangle$ d'où le résultat
par différence.

\medskip
ii)
Il est clair que $\uQ(\Un) = 0$ donc $\langle\uQ\rangle^\sat \subseteq I$; quant à l'inclusion
opposée (ici intervient $\delta \ne 0$) pour tout $|\gamma| = \delta-1$,
on a $X^\gamma (X_i - X_j) \in \langle\uQ\rangle$ d'après le point i)  donc $X_i -
X_j \in \langle\uQ\rangle^\sat$.
\end {proof}

\index{hauteur!sur!2@$\bbN^n_\delta$}%

\subsubsection*{Le jeu généralisé $\uQ=\pXD+\bsq\,\uR$
($\uR$ jeu monomial de format $D$ ayant une fonction hauteur)}

\index{jeu!monomial!de format $D$}%

Ici, nous changeons les notations pour $\uQ$ en désignant par $\uQ$ le système
$\pXD + \bsq\,\uR$ dans lequel interviennent~$2n$ indéterminées sur
$\bbZ$ notées $p_1, \dots, p_n, q_1, \dots, q_n$:
$$
\uQ = \pXD + \bsq\,\uR,\qquad    Q_i \ =\ p_i X_i^{d_i} + q_i R_i
$$
On suppose que \fbox{$\uR$ admet une fonction hauteur (sous-entendu sur
$\bbN^n_\delta$)}. On note $(\ell_1, \dots, \ell_n)$ les $n$ lignes
de la matrice $\Jac_\Un(\uX^D-\uR)$ attachée à $\uR$ et
$(r_1, \dots, r_n)$ ses $n$ mineurs principaux.

De manière analogue à la proposition~\ref{OmegaXalpha4blocs}, 
on peut déterminer $\omega_\uQ$ via ses évaluations $\omega_\uQ(X^\alpha)$.
En notant $\overline {q}_j = -q_j$, cette évaluation est
un monôme en les $p_i, \overline{q}_j$, produit d'une $q$-contribution et d'une $p$-contribution:
$$
\forall\, |\alpha| = \delta, \quad 
\omega_\uQ(X^\alpha) \ = \ 
\prod_{X^\beta \in \calO_\alpha \setminus \{X^\emouton\}} \kern -9pt \overline {q}_{\minDiv(X^\beta)}
\prod_{X^\gamma \in \overline{\calO_\alpha}} p_{\minDiv(X^\gamma)}
$$
Au vu de notre étude, point v) du corollaire~\ref{CorollaireExistenceHauteur}, nous
pouvons apporter une précision importante concernant la $q$-contribution.

\label{NOTA06-ovq}%
%
%

\begin {theo}[Caractère intrinsèque et interprétation de la $q$-contribution] \label{qContributionTheorem}
\leavevmode

\begin {enumerate}[\rm i)]
\item
La $q$-contribution de $\omega_\uQ(X^\alpha)$
est \emph{intrinsèque} à $X^\alpha$ i.e. ne dépend pas de la sélection
$\minDiv$ qui pilote $\omega_\uQ$.

\item
De manière précise, en écrivant la formule sous la forme suivante avec des
exposants anonymes pour les $p_i$:
$$
\omega_\uQ(X^\alpha) = \overline{q}_1^{m_1}\cdots\overline{q}_n^{m_n}\  p_1^\sbullet\cdots p_n^\sbullet
$$
alors, dans tout chemin du graphe de $\uR$ reliant $X^\alpha$ au \MoutonNoir{} $X^\emouton$, le
nombre d'étiquettes~$i$ est égal à $m_i$.  De plus, les $m_i$ sont ceux qui
interviennent dans l'unique écriture positive minimale de $\alpha - (\emouton)$ sur les $\ell_i$:
$$
\alpha - (\emouton) = \sum_i m_i \ell_i, \qquad\quad
h(\alpha) = \sum_i m_i
$$
\item
Soit une écriture $\alpha - (\emouton) = \sum_i m_i \ell_i$ où les $m_i$ sont dans $\bbN$.
C'est l'écriture positive minimale de $\alpha - (\emouton)$ sur les $\ell_i$
si et seulement si il existe un $j$ tel que $m_j < r_j$. 
\end{enumerate}
\end {theo}

\index{graphe (orienté étiqueté) d'un jeu monomial}%

\begin {proof} \leavevmode

i) découle de ii) qui découle du 
point v) du corollaire~\ref{CorollaireExistenceHauteur}.

\medskip

iii) 
On rappelle que les $r_i$ sont $\geqslant 0$ et que $\pgcd(r_1, \dots, r_n) = 1$.

$\rhd$
Supposons $m_j < r_j$ pour un $j$. 
Soient des $m'_i \in \bbN$ tels que
$\sum_i m'_i \ell_i =   \sum_i m_i \ell_i$. Il existe $q \in \bbZ$ tel que
$m'_\sbullet - m_\sbullet = q\, r_\sbullet$. Comme $m'_j \geqslant 0$ et $0 \leqslant m_j < r_j$,
l'égalité $m'_j = qr_j + m_j$ prouve que $q \geqslant 0$. De $m'_i = qr_i + m_i$, on 
tire $m'_i  \geqslant m_i$ pour tout $1 \leqslant i \leqslant n$.

$\rhd$ Supposons $m_i \geqslant r_i$ pour tout $i$. On pose $m'_i = m_i - r_i$.
Alors $m'_i \in \bbN$ et $\sum_i m_i \ell_i = \sum_i m'_i \ell_i$.
Et on~a $\sum_i m'_i < \sum_i m_i$ car il y a au moins un $r_i \ne 0$.
Donc $\sum_i m_i \ell_i$ n'est pas minimale.
\end {proof}

\subsubsection {Pour un jeu monomial $\uR$: existence d'une fonction hauteur versus simplicité}

\begin {defn} [Jeu simple] \label{DefJeuSimple}
Convenons de dire qu'un jeu monomial $\uR$ de format $D$ est \emph {simple} si le
système $\uQ=\uX^D-\uR$
admet comme  seul zéro projectif le point $\Un = (1 : \cdots : 1)$
et si ce dernier est simple. Ou encore, de manière équivalente,
si le saturé $\langle\uQ\rangle^\sat$ est l'idéal $I = \langle X_i-X_j, 1 \leqslant 
i,j \leqslant n\rangle$, idéal du point~$\Un$.
\end {defn}

\index{jeu!monomial!simple}%

Une manière équivalente de décréter que $\uR$ est simple consiste à
faire intervenir la série d'Hilbert-Poincaré de la $\bbQ$-algèbre
graduée $\bbQ[\uX]/\langle\uQ\rangle$ où $\uQ=\uX^D-\uR$:
$$
S_{\bbQ[\uX]/\langle\uQ\rangle} \overset{\rm def.}{=}
\sum_{k\ge 0} \dim\big(\bbQ[\uX]_k/\langle\uQ\rangle_k\big)\, t^k
$$
C'est une fraction rationnelle de la forme
$$
S_{\bbQ[\uX]/\langle\uQ\rangle}(t) = \dfrac{N(t)}{(1-t)^d},
\qquad\quad
N(t) \in \bbZ[t],\quad N(1) \ge 1,\quad  d \ge 1
$$
En des termes algébriques: l'entier $d$ est la dimension de Krull de
l'anneau $\bbQ[\uX]/\langle\uQ\rangle$ (cette dimension est $\ge 1$
car $\langle\uQ\rangle$ est contenu dans l'idéal engendré par tous les
$X_i - X_j$) et $N(1)$ est ce que l'on appelle le degré de la
$\bbQ$-algèbre graduée $\bbQ[\uX]/\langle\uQ\rangle$.  De manière
géométrique, l'entier $d-1$ est la dimension de la variété projective
$\{\uQ=0\}$ définie sur $\bbQ$ et l'entier $N(1)$ représente le degré
de cette variété projective. Dire que $\uR$ est simple se traduit donc
par $d=1, N(1)=1$.

\medskip

Après de nombreuses expérimentations, nous proposons la conjecture suivante, qui énonce
une réciproque à la proposition~\ref{HauteurSimplicite}.

\begin {conj} \label {ConjectureJeuSimple} \leavevmode
Soit $\uR$ un jeu monomial simple. Alors il existe une fonction hauteur pour $\uR$.
\end {conj}

Pour mesurer l'impact de ce résultat conjectural, reprenons l'exemple du jeu $\uR^{(1)}$
privilégiant $X_1$, défini par $R_i = X_1^{d_i-1} X_{i+1}$, où
l'indexation est cyclique modulo~$n$ ($X_{n+1} = X_1$). En ce qui
concerne sa simplicité, nous laissons le
soin au lecteur d'examiner, sur un anneau quelconque, le système
projectif $\uQ(\ux) = 0$
$$
\setlength{\arraycolsep}{0.5\arraycolsep}
\left\{
\begin {array} {ccccc}
x_1^{d_1} &=& R_1(\ux) &=& x_1^\bullet x_2 \\
x_2^{d_2} &=& R_2(\ux) &=& x_1^\bullet x_3 \\
         &\vdots&     &\vdots&   \\
x_{n-1}^{d_{n-1}} &=& R_{n-1}(\ux) &=& x_1^\bullet x_n \\
x_n^{d_n} &=& R_n(\ux) &=& x_1^\bullet x_1 \\
\end {array}
\right.
$$
et de contrôler que, dès qu'un $x_i$ est inversible, il en est de même
de $x_{i+1}$ donc de tous. Puis d'en déduire que $x_1 = x_2 = \cdots =
x_n$.  On perçoit ici la facilité à vérifier la simplicité de $\uR$;
cette facilité est à mettre en opposition avec la
complexité à montrer l'existence d'une fonction hauteur pour $\uR$.

\bigskip

Soit $\uR$ un jeu monomial de format $D$ et $\Jac_\Un(\uX^D-\uR)$ sa
matrice attachée; notons $\ell_1,\dots,\ell_n \in \bbZ_0^n$ ses~$n$
lignes et $(r_1, \cdots, r_n)$ ses mineurs principaux. Pour ces jeux
$\uR$, on a dégagé 3 catégories schématisées ci-après par des
rectangles d'inclusion entre ces catégories. On aurait
bien aimé disposer pour la catégorie intermédiaire (propriété d'existence
d'une fonction hauteur sur $\bbN_\delta^n$), d'une caractérisation de nature aussi
simple que celle figurant dans le rectangle interne.

\medskip

\begin {tikzpicture}
\coordinate (A1) at (0,0) ;
\coordinate (X1) at (15,0) ;
\coordinate (Y1) at (0,4) ;
\draw[rounded corners = 3mm] (A1) rectangle ($(X1)+(Y1)$) ;
\draw (Y1) node [below right]
   {$\uR$ vérifie $\bbZ_0^n=\bbZ\ell_1+\cdots+\bbZ\ell_n \iff \pgcd(r_1,\dots,r_n)=1$} ;
\coordinate (A2) at (0.5,0.5) ;
\coordinate (X2) at (14.2,0) ;
\coordinate (Y2) at (0,2.5) ;
\draw[rounded corners = 3mm] (A2) rectangle ($(A2)+(X2)+(Y2)$) ;
\draw ($(A2)+(Y2)$) node [below right] {$  
  \begin{array} {l}
  \text{existence pour $\uR$ d'une fonction hauteur sur $\bbN^n_\delta$} 
  \iff \text{(conjecturalement) simplicité de $\uR$} \\
  \end {array}$} ;
\coordinate (A3) at (1.5,1) ;
\coordinate (D3) at (12,2) ;
\draw[rounded corners = 3mm] (A3) rectangle (D3) ;
\draw (1.5,2) node [below right] {$
  \begin{array} {l}
    \text{existence pour $\uR$ d'une fonction hauteur sur } \bbZ^n_0 \iff \\
    \pgcd(r_1,\dots,r_n)=1 \text{ et } r_i \ge 1 \ \forall i
    \iff \bbZ_0^n = \bbN\ell_1+\cdots+\bbN\ell_n\\
  \end {array}$} ;
\end {tikzpicture}

\medskip
En vrac, quelques chiffres pour illustrer le fait que le nombre de jeux simples $\uR$
n'est pas négligeable. Le nombre de jeux monomiaux $\uR$ de format
$D = (d_1, \dots, d_n)$, en excluant $R_i = X_i^{d_i}$ pour éviter les
boucles dans le graphe de $\uR$, est donné par le produit $N$ suivant:
$$
N = \prod_{i=1}^n  (N_i-1) ,  \qquad  N_i = \binom{d_i + n-1}{n-1}
$$
Par exemple, pour $D = (2,2,3,4)$, on a $N = 9 \times 9 \times
19\times 34 = 52326$ jeux monomiaux de format $D$ dont $N' := 12 557$
sont simples. Dans le tableau suivant, $N$ désigne le nombre de jeux de
format $D$, et $N'$ le nombre de jeux simples.
$$
\begin{array}{c|c|c|c|c|c|c|c|c|c|c|}  
D &(2,2,3)&(2,2,4)&(2,3,3)&(2,3,4)&(2,3,5)&(3,3,3)&(3,3,4)&(1,2,2,2)&(1,2,2,3) \\
\hline 
N  &225   &350    &405    &630    &900    &729    &1134   &2187      &4617   \\
\hline 
N' &83    &128    &161    &244    &381    &305    &458    &540       &1300    \\
\end {array}
$$

\subsubsection {Pas toujours de phénomène bonus sur $\Jex_{2,d}$ pour $d\geqslant\delta+1$}

Soit $R = (X_2^2, X_1, X_2)$ de format $(2,1,1)$ de degré critique 1. Le graphe de $R$
est $X_3 \to X_2 \to X_1$, sans circuit d'où l'existence
d'une fonction hauteur. Le graphe défini de manière analogue au graphe
de~$\uR$ mais sur $\Jex_{2,d}$ avec $d=3$ présente un circuit:
$$
\xymatrix {
\vertex {X_1^2X_2} &\vertex{X_1X_2X_3}\ar[ld]^2   &\vertex {X_2X_3^2} \ar[d]^3 \\
\vertex {X_1^2 X_3}\ar[u]^3 \ar[rr]_1         &&\vertex{X_2^2X_3} \ar[ul]^2 \\
}
$$
En posant $\uQ = \pXD + \bsq\,\uR$, on a
$$
W_{2,3}(\uQ) =
\EastBordermatrix{
p_{1} & . & . & . & . & \Heti{X_{1}^{2}X_{2}} \\ 
. & p_{1} & q_{2} & . & . & \Heti{X_{1}^{2}X_{3}} \\ 
. & . & p_{2} & q_{2} & . & \Heti{X_{1}X_{2}X_{3}} \\ 
. & q_{1} & . & p_{2} & . & \Heti{X_{2}^{2}X_{3}} \\ 
. & . & . & . & p_{2} & \Heti{X_{2}X_{3}^{2}} \\ 
}
$$
dont le déterminant n'est pas un monôme en les $p_i$:
$$
\det W_{2,3}(\uQ) = p_1p_2(p_1p_2^2 + q_1q_2^2) = p_1^2p_2^3 + p_1p_2q_1q_2^2
$$
Ici, pas de phénomène bonus pour $\Jex_{2,d}$ avec $d=3 > \delta=1$.

\cleardoublepage

\section{Propriétés de régularité de $\bfB=\bfA[\protect\uX]/\langle\protect\uP\rangle$ en générique}
\label{ChapBgenerique}

Ici, \emph {sauf mention expresse du contraire}, le terrain est
générique: l'anneau de base est un anneau de polynômes $\bfA =
\bfk[\indetsPi]$ en des indéterminées allouées aux coefficients des
polynômes $P_i$ du système $\uP$.
Nous allons étudier les points suivants (qui ne sont pas indépendants):

\begin {enumerate}[1.]
\item
Structure du localisé $\bfB_{x_i}$.
\item
Régularité du scalaire $\det W_{1,\delta}$ sur $\bfB_\delta$.
\item[2'.]
Injectivité de la forme linéaire $\overline\omega : \bfB_\delta \to \bfA$.
\item
Régularité du scalaire $\det W_{1,\delta}$ sur $\bfA[\uX]/\uPsat$.
\end {enumerate}

\label{NOTA07-Bdelta}%
\label{NOTA07-omegaBdelta}%

\medskip

Rappelons notre terminologie : étant donné un $\bfA$-module $M$, un
scalaire $a \in \bfA$ est dit régulier sur~$M$ ou bien $M$-régulier si,
pour $m \in M$, on a l'implication $am = 0 \implies m = 0$.
Lorsque~$M$ est fidèle \idest{} $\Ann(M) = 0$, tout scalaire $a$ qui est
$M$-régulier est régulier comme le montrent les implications pour $a' \in \bfA$:
$$
aa' = 0 \quad\Rightarrow\quad a(a'M) = 0
\quad\overset{\text{$a$ $M$-régulier}}{\Longrightarrow}\quad a'M = 0
\quad\overset{\Ann(M)=0}{\Longrightarrow}\quad a' = 0
$$
Rappelons également que lorsque $\uP$ est régulière,
le $\bfA$-module $\bfB_\delta$ est fidèle (conséquence du fait que le déterminant
bezoutien $\overline\nabla \in \bfB_\delta$ est sans torsion, cf le point i) de la
proposition \ref{MiniWiebe}).

\medskip

Voici un rapide aperçu de ce chapitre (nous omettons le premier point qui est un peu à part).

\subsubsection*{Objet du point 2.}

Pour $a = \det W_{1,\delta}(\uP)$, la propriété de régularité en degré
$\delta$ mentionnée au point 2. équivaut à:
$$
\forall\, F \in \bfA[\uX]_\delta, \qquad 
aF \in \langle\uP\rangle_\delta
\ \implies \ 
F \in \langle\uP\rangle_\delta
$$
Cette propriété \og $\det W_{1,\delta}$ est $\bfB_\delta$-régulier\fg{}
est une \emph {propriété du jeu générique}, qui n'est pas
impliquée par le seul fait que $\uP$ soit régulière.

\medskip

Considérons par exemple le jeu étalon généralisé $\pXD = (p_1X_1^{d_1}, \dots, p_nX_n^{d_n})$
et montrons que $a := \det W_{1,\delta}(\pXD )$ n'est pas $\bfB_\delta$-régulier.
Ce déterminant est un monôme en $p_1,\dots, p_n$
multiple du produit $\pi = p_1 \cdots p_n$.
Prenons  un monôme $X^\alpha$ de degré $\delta$ autre que le \MoutonNoir; il y a donc
un indice $i$ tel que $\alpha_i \ge d_i$ de sorte que $\pi X^\alpha \in \langle\pXD\rangle_\delta$.
A fortiori $aX^\alpha \in \langle \pXD \rangle_\delta$ sans avoir $X^\alpha \in
\langle\pXD\rangle_\delta$. Ainsi $a$ est un scalaire régulier non
$\bfB_\delta$-régulier. De la même manière, aucun des $p_i$ n'est $\bfB_\delta$-régulier.
Pour $p_1$ par exemple, prenons $F = p_2\cdots p_n\, X_1^\delta$; alors $p_1F \in 
\langle\pXD\rangle_\delta$ sans avoir $F \in \langle\pXD\rangle_\delta$.

\smallskip

En liaison avec ce qui suit, en notant $\omega$ pour $\omega_{\pXD}$,
on dispose d'une inclusion \emph{stricte} $\langle\pXD\rangle_\delta
\subsetneq \Ker\omega$.  En effet, on a $\omega = \det
W_{1,\delta}\times (X^\emouton)^\star$ (cf la
section~\ref{soussection-d-egal-delta}, en particulier en
page~\pageref{omegaCasParticulier}) de sorte que $\Ker\omega =
\Jex_{1,\delta} = \langle\uX^D\rangle_\delta$.

\subsubsection*{Autour du point 2'.}

Rappelons que la forme linéaire $\overline\omega : \bfB_\delta \to
\bfA$ est toujours définie pour une suite $\uP$ quelconque, ce qui est dû au fait
que $\omega : \bfA[\uX]_\delta \to \bfA$ est nulle sur $\uPdelta$
\idest{} $\uPdelta \subseteq \Ker \omega$.  Le cas d'égalité $\uPdelta
= \Ker \omega$, qui équivaut à l'injectivité de
$\overline\omega$, va jouer un rôle important dans la suite de notre
étude.

Le schéma suivant fournit un lien étroit entre les propriétés des points {2.} et {2'.}:
$$
\xymatrix @R=1.2cm @C=0.3cm{
     &\det W_{1,\delta} \text{ est $\bfB_\delta$-régulier} \ar@{=>}[dr]
\\
\txt{tout scalaire régulier\\ est $\bfB_\delta$-régulier}
        \ar@{=>}[ur]|{\txt{sous couvert de\\ $\det W_{1,d}$ régulier}}
   &
   &\overline{\omega}:\bfB_\delta\to\bfA \text { est injective} \ar@{=>}[dl]
\\
   &\txt{le $\bfA$-module $\bfB_\delta$\\ se plonge dans $\bfA$} \ar@{=>}[ul]
\\   
}
$$
Il apporte également l'éclairage structurel sur la possibilité, dans un contexte adéquat,
de plonger $\bfB_\delta$ dans~$\bfA$


Sous quelles conditions les \emph{implications} représentées par ce
diagramme sont-elles valides? Réponse: aucune contrainte sur $\uP$
n'est nécessaire. Et c'est pour cette raison que nous avons mentionné
au tout début le \emph{sauf mention expresse du contraire} 
car nous nous réservons la possibilité d'étudier par endroits
certaines imbrications entre propriétés sans supposer $\uP$ nécessairement
générique. Ces quatre implications sont de surcroit aisées sauf peut-être
l'implication \rotatebox[y=1.5mm]{-45}{$\Rightarrow$}.
Nous l'avons déjà utilisée pour le jeu circulaire (voir le second
point de la remarque figurant après la proposition
\ref{ProprietesJeuCirculaire}). Reprenons l'argument pour
souligner que tout repose sur la propriété Cramer de $\omega$: $\omega(G)F - \omega(F)G
\in \uPdelta$ pour $F, G \in \bfA[\uX]_\delta$. Pour $F \in
\bfA[\uX]_\delta$ tel que $\omega(F) = 0$, on doit montrer que $F \in
\uPdelta$. En particulier, pour $G = X^\emouton$, on a
$$
\omega(X^\emouton)F \ - \ \underbrace{\omega(F)}_{=0} X^\emouton \ \in \ \uPdelta
\qquad \text{d'où } \qquad 
\omega(X^\emouton)F \ \in \ \uPdelta
$$
Or, par hypothèse, $\omega(X^\emouton) = \det W_{1,\delta}$ est $\bfB_\delta$-régulier,
donc $F \in \uPdelta$.

\label{BdeltaRegImpliesInjectivity}

\medskip

Ce diagramme d'implications n'a d'intérêt que si l'on dispose de
critères permettant d'assurer les conditions figurant en ses \emph
{sommets}.  Cette tâche est beaucoup plus délicate et le
statut du jeu~$\uP$ va se révéler primordial.

\subsubsection*{Point 3.: régularité de 
$\det W_{1,\delta}(\uP^\gen)$ modulo $\langle\uP^\gen\rangle^\sat$
et intervention du jeu circulaire}

Ainsi dans le cas du jeu générique, la clause figurant sur le sommet
haut du diagramme est vérifiée donc les clauses aux autres sommets le
sont également. Cette propriété de $\bfB_\delta$-régularité du
scalaire~$\det W_{1,\delta}$, nous la déduirons (cf le
lemme~\ref{PsatRegImpliesBdeltaReg}) d'une propriété beaucoup plus
forte: ce scalaire est en fait régulier sur $\bfA[\uX]/\uPsat$, cf le
théorème~\ref{GenPsatRegScalars}.
En ce sens, $\det W_{1,\delta}$ se distingue
notablement d'autres scalaires réguliers de $\bfA$, qui sont également
$\bfB_\delta$-réguliers sans être nécessairement réguliers sur
$\bfA[\uX]/\uPsat$: considérer par exemple un élément régulier de
$\bfA$ appartenant à l'idéal d'élimination $\uPsat\cap\bfA$!

\medskip

Dans cette introduction, il nous est difficile de préciser comment nous allons utiliser
le jeu circulaire, cf les sections \ref{SectPsatRegBdeltaRegI} et
\ref{SectPsatRegBdeltaRegII}.

\subsubsection*{Le statut du couple $(\det W_{1,\delta},\ \omega)$ versus
$(\det W^\sigma_{1,\delta},\ \omega^\sigma)$. Importance du degré critique $\delta$.}

$\blacktriangleright$
Le mouton-noir $X^\emouton$ relie le scalaire $\det W_{1,\delta}$ à la forme
linéaire $\omega : \bfA[\uX]_\delta \to \bfA$ via $\det W_{1,\delta} =
\omega(X^\emouton)$ et les deux objets (déterminant, forme) sont pilotés par 
le même mécanisme de sélection $\minDiv$.
Dans l'histoire de régularité sur $\bfB_\delta$ ou sur $\bfA[\uX]/\uPsat$,
ce mécanisme de sélection n'a aucune vertu particulière et les résultats
s'adaptent aux versions  $(\det W_{1,\delta}^\sigma,\ \omega^\sigma)$
tordues par $\sigma \in \fS_n$.

\medskip

Nous définirons ultérieurement (cf. le chapitre~\ref{ChapMacRaeForP}) une forme
linéaire \emph{intrinsèque} $\omegares : \bfA[\uX]_\delta \to \bfA$, diviseur des
formes $\omega^\sigma$, qui possède la propriété du point {2'.} en terrain \emph {générique}:
$$
\Ker\omegares = \langle\uP\rangle_\delta, \qquad\qquad
\uomegares : \bfB_\delta \hookrightarrow \bfA \text{ est injective}
$$
Cette forme $\omegares$,  épine dorsale de notre étude, est étroitement liée
aux formes $(\omega^\sigma)_{\sigma\in\fS_n}$: dans dans un contexte adéquat, elle en est
le pgcd fort.

Mais pour l'instant, elle n'est pas en place. Nous avons seulement
les $\omega^\sigma$ à notre disposition et pour cette raison, nous
énonçons les résultats avec $\omega$ et $\det W_{1,\delta} = \omega(X^\emouton)$.

\medskip

$\blacktriangleright$
Pour $d \geqslant \delta+1$, le comportement de $\det W_{1,d}(\uP)$
sur $\bfB_d$ est tout autre: pour tout $|\alpha| = d$, on a $\det
W_{1,d}(\uP)\, X^\alpha \in \langle\uP\rangle_d$ ce qui s'énonce aussi
$\det W_{1,d} \in \Ann(\bfB_d)$ (et ceci pour n'importe quel
système~$\uP$), en totale opposition avec la $\bfB_d$-régularité.

\subsubsection*{A propos des diverses notions de régularité}

L'étude qui vient met en jeu diverses propriétés de régularité parmi lesquelles
intervient le localisé~$\bfB_{x_i}$ où $x_i$ est la classe de $X_i$ dans $\bfB$
(on rejoint ainsi le point {1.}).

A ce propos, il est important de remarquer que, d'une part, les
localisations $\bfB \to \bfB_{x_i}$ ne sont pas injectives et que
d'autre part, $\bfB$ est un quotient de $\bfA[\uX]$. Il convient donc
de ne pas mélanger, pour un élément $a \in \bfA$, les notions de
régularité sur $\bfA, \bfB, \bfB_{x_i}, \bfB_{\delta}, \bfB_{d}$.
Considérons par exemple un système générique $\dsL$ de $n$ formes
linéaires.  Le degré critique $\delta$ est nul et si on prend $a =
\det(\dsL)$, alors $a$ est régulier dans $\bfA$, mais n'est pas
régulier dans~$\bfB_{x_i}$ puisque $ax_i=0$ sans que $x_i$ ne soit
nul; il est régulier sur $\bfB_0$ mais pas sur $\bfB_1$.

\medskip

Enfin, comme déjà mentionné, nous n'avons pas souhaité nous limiter systématiquement
au cadre générique d'où une certaine attention demandée au lecteur. Nous préciserons avec le plus
grand soin le contexte utilisé localement.

\subsection{Structure du localisé $\bfB_{x_j}$ (en générique ou assimilé)}

\label{NOTA07-Bxj}%

Notons $\bfA'_j$ la $\bfk$-algèbre engendrée par les coefficients des~$P_i$ 
autres que le coefficient~de~$X_j^{d_i}$ dans chaque $P_i$. 
On va montrer que l'on dispose d'un isomorphisme canonique de $\bfk$-algèbres 
$$
\bfB_{x_j} \ \simeq \ 
\bfA'_j[\uX]_{X_j}
$$
C'est l'isomorphisme qui provient de la localisation en $X_j$ du morphisme 
$\bfA'_j[\uX] \longrightarrow \bfB$.

En particulier, $\bfB_{x_j}$ est isomorphe en tant que $\bfk$-algèbre au localisé 
d'une $\bfk$-algèbre de polynômes.

\medskip

Afin d'alléger les notations, on énonce le résultat avec $j=n$ ce qui permet 
également d'y voir plus clair.
On adopte les notations suivantes:
$p_{i,n}$ désigne le coefficient en $X_n^{d_i}$ de $P_i$ et 
$\bfA'$ la $\bfk$-algèbre engendrée par les coefficients des~$P_i$ 
autres que les~$p_{i,n}$.
Ainsi $\bfA = \bfA' [p_{1,n}, \dots, p_{n,n}]$ 
et $P_i = p_{i,n} X_n ^{d_i} + P_i'$ avec ${P'_i \in \bfA'[\uX]}$.
Ceci a l'avantage de mettre en évidence un \emph {nouveau contexte}
dans lequel $\uP$ n'est pas nécessairement générique: il suffit que les coefficients
$p_{i,n}$ soient algébriquement indépendants sur un anneau de base.
De manière précise, il suffit de disposer d'un
anneau de polynômes $\bfA = \bfA'[p_{1,n},\cdots,p_{n,n}]$ et que chaque
$P_i$ soit à coefficients dans $\bfA'$ à l'exception du coefficient
$p_{i,n}$ en $X_n^{d_i}$ de $P_i$.

Le lecteur pourra d'ailleurs vérifier que la proposition suivante
s'applique même si $P'_i = 0$ i.e. pour 
$\uP = (p_{1,n}X_n^{d_1}, p_{2,n}X_n^{d_2}, \cdots, p_{n,n}X_n^{d_n})$ qui est loin
d'être générique! (ce n'est même pas une suite régulière).
Il s'agit certes d'un exemple un tantinet pathologique mais en voici
un autre qui ne l'est pas. Soit $L_1, \cdots, L_n$ un système
générique de $n$ formes linéaires en $n$ variables $X_1, \cdots, X_n$
et, pour un exposant $q \ge 1$, le système $\uP$ défini par:
$$
P_i = L_i(X_1^q, \cdots, X_n^q)
$$
Alors $\uP$ n'est pas générique mais on peut tout de même lui appliquer
le résultat suivant.

\label{NOTA07-pij}%
%
%

\begin{prop}\label{StructureLocalise}

Dans le contexte ci-dessus (qui couvre le cas de $\uP$ générique), le
morphisme canonique $\bfA' [\uX]_{X_n} \rightarrow \bfB_{x_n}$ est un
isomorphisme.
  
\smallskip
\noindent
$\rhd$ L'injectivité se traduit par $\bfA' [\uX] \cap \langle \uP \rangle = 0$.

\smallskip
\noindent
$\rhd$ La surjectivité s'explicite de la manière suivante :
pour $H \in \bfA[\uX]$, il existe
$H' \in \bfA'[\uX]$ tel que $X_n^e H \equiv H' \bmod \langle \uP \rangle$ 
pour un certain exposant $e$ ;
l'image de 
$H' / X_n^e \in \bfA' [\uX]_{X_n}$ dans $\bfB_{x_n}$ est alors $\overline H$.
\end{prop}

\begin{proof} \leavevmode

Injectivité. Puisque le morphisme de localisation $\bfA'[\uX] \to
\bfA'[\uX]_{X_n}$ est injectif, l'injectivité du morphisme
$\bfA'[\uX]_{X_n} \to \bfB_{x_n}$ est équivalente à celle de
$\bfA'[\uX] \to \bfB_{x_n}$.  A quelle condition $F \in \bfA'[\uX]$
est-il dans le noyau de $\bfA'[\uX] \to \bfB_{x_n}$? C'est le cas si
et seulement si, pour un certain exposant $e$, on a $x_n^e \overline F
= 0$ \idest{} $X_n^e F \in \langle\uP\rangle$.
Formellement:
$$
\Ker(\bfA'[\uX] \to \bfB_{x_n}) = \{ F \in \bfA'[\uX] \mid \exists\ e
\text{ tel que } X_n^e F \in \langle \uP \rangle \}
$$
On en déduit que la nullité du noyau à gauche est équivalente à celle de 
$\bfA' [\uX] \cap \langle\uP\rangle$.

Posons alors $\bfR = \bfA'[\uX]$, notons $U_i := p_{i,n}$ pour alléger et voyons $P_i$ 
comme polynôme de $\bfR[U_1, \dots, U_n]$ via l'écriture $P_i = r_i U_i + s_i$
avec $r_i = X_n^{d_i}$ et $s_i = P'_i$. 
L'égalité $\bfA' [\uX] \cap \langle \uP \rangle = 0$ provient du résultat suivant :
\begin{quote}
Soit $r_1, \dots, r_n, s_1, \dots, s_n$ des éléments d'un anneau $\bfR$ et 
$U_1, \dots, U_n$ des indéterminées sur $\bfR$.
Si les $r_i$ sont réguliers, alors $\bfR \cap \langle r_1 U_1 + s_1, \dots, r_n U_n + s_n\rangle = 0$.
\end{quote}
Ce résultat se montre en remarquant tout d'abord que 
l'anneau $\bfR$ s'injecte dans le localisé $\bfR_{r_1 \cdots r_n}$ (les~$r_i$ sont réguliers) 
puis en réalisant l'évaluation $U_i := -s_i/r_i$ dans ce localisé.

\medskip

Surjectivité.
On a $\bfB = \bfA[\uX]/\langle \uP \rangle = \bfA'[\uU][\uX] / \langle \uP \rangle 
= \bfA'[\uX][\uU] /\langle \uP \rangle$.
Il reste donc à voir pourquoi $\overline U_i \in \bfB$ est dans
l'image de $\bfA'[\uX]_{X_n} \rightarrow \bfB_{x_n}$.
Ceci est dû à l'égalité $P_i = U_i X_n^{d_i} + P'_i$ 
qui prouve que $\overline U_i$ est l'image de  $- {P'_i} / X_n^{d_i}$.
\end{proof}

\begin{rmq}
On peut fournir la même preuve mais de manière plus concise.

Pour un anneau $\bfS$, des indéterminées $U_i$ sur cet anneau et des $t_i \in \bfS$, 
il est clair que le morphisme canonique suivant est un isomorphisme 
$$
\bfS \ \longrightarrow \ 
\dfrac{\bfS[U_1, \dots, U_n]}{\langle U_1 + t_1, \dots, U_n + t_n \rangle}
$$
En appliquant cette remarque à $\bfS=\bfA'[\uX]_{X_n}$, $U_i=p_{i,n}$,
et $t_i = \frac{P'_i}{X_n^{d_i}}$ où $P'_i = P_i-p_{i,n}X_n^{d_i}$ de
sorte que $P_i = (U_i + t_i) X_n^{d_i}$, l'anneau quotient de droite
est exactement
$$
\dfrac{\bfA'[\uX]_{X_n}[U_1, \dots, U_n]}{\langle P_1/X_n^{d_1}, \dots, 
P_n/X_n^{d_n} \rangle}
\ = \ 
\dfrac{\bfA'[U_1, \dots, U_n][\uX]_{X_n}}{\langle P_1, \dots, P_n \rangle }
\ = \ 
\bfA[\uX]_{X_n}/\langle\uP \rangle 
\ = \ 
\bfB_{x_n}
$$

%
\end{rmq}

\begin{prop}[En générique, pour saturer, un seul localisé suffit] \label{GeneriqueLocaliseSature} 
Soit $\uP$ la suite générique au dessus d'un anneau $\bfk$.

\begin{enumerate}[\rm i)]
\item 
Chaque $x_i$ est un élément régulier du localisé $\bfB_{x_j}$. 

\item
Le noyau de chaque morphisme canonique $\bfA[\uX] \to \bfB_{x_i}$ est l'idéal $\uPsat$,
idéal saturé de $\langle\uP\rangle$.

\item
L'anneau $\bfA[\uX]/\uPsat$ s'injecte dans $\bfA'_i[\uX]_{X_i}$,
localisé en $X_i$ de l'anneau de polynômes $\bfA'_i[\uX]$ sur $\bfk$.
En particulier, si $\bfk$ est intègre, le saturé $\uPsat$ est un idéal premier de $\bfA[\uX]$ et
l'idéal d'élimination $\uPsat \cap \bfA$ est un idéal premier de $\bfA$.
\end{enumerate}
\end{prop}

\index{théorème!de saturation de $\langle\uP\rangle$ en générique}%

\begin{proof}
i) 
D'après~\ref{StructureLocalise},
$\bfB_{x_j}$ est le localisé d'un anneau de polynômes où $X_i$ figure comme indéterminée.
Cet élément $X_i$ reste régulier dans ce localisé.

ii)
Soit ${F \in \bfA[\uX]}$ nul dans le localisé $\bfB_{x_i}$. 
Montrons que $F$ est nul dans tous les localisés~$\bfB_{x_j}$.
Contemplons le rectangle commutatif dans lequel les morphismes sont les
morphismes de localisation. Le morphisme vertical à droite est injectif 
d'après i).
$$
\xymatrix {
\bfA[\uX] \ar[d]\ar[r]           &\bfB_{x_j}\ar[d]^{\rm inj.} \\
\bfB_{x_i}\ar[r]            &\bfB_{x_jx_i} \\
}
\hspace{2cm}
\xymatrix {
F \ar[d]\ar[r]           & F_{x_j}\ar[d]^{\rm inj.} \\
F_{x_i}\ar[r]            & F_{x_jx_i} \\
}
$$
Par hypothèse $F_{x_i} = 0$, a fortiori $F_{x_ix_j} = 0$.
L'élément $F_{x_j}$ s'envoie sur $0$ par 
le morphisme \textit{injectif}, donc $F_{x_j} = 0$.

iii) D'après ii), l'anneau $\bfA[\uX]/\uPsat$ s'injecte dans $\bfB_{x_i}$, isomorphe 
à $\bfA'_i[\uX]_{X_i}$ (cf.~\ref{StructureLocalise}).
Supposons~$\bfk$ intègre. Alors l'anneau $\bfA'_i$ est intègre, 
$\bfA'_i[\uX]$ aussi, son localisé en $X_i$ également.
On en déduit que l'idéal $\uPsat$ est premier, et il en est de même
de sa trace sur $\bfA$.
\end{proof}

\begin {rmq}\leavevmode

$\blacktriangleright$  
Soit, dans le contexte de la proposition précédente, un polynôme $F
\in \bfA[\uX]$.  Le point ii) affirme donc qu'une appartenance $X_i^e
F \in \langle\uP\rangle$ pour un certain indice $i$ et un certain
exposant~$e$ a comme conséquence que, pour tout indice $j$, il y a un
exposant $e_j$ tel que $X_j^{e_j} F \in \langle\uP\rangle$. On se
gardera de croire que l'on peut prendre $e_j = e$. Voir l'exemple
ci-après.

\medskip

$\blacktriangleright$
Par contre, en degré $\delta$, il y a un phénomène particulier qui est le suivant.
Pour $F \in \bfA[\uX]_\delta$, en terrain générique bien entendu,
on dispose de l'implication:
$$
\text{$X_i^e F \in \langle\uP\rangle$ pour un certain $i$ et un certain $e$}
\qquad \Longrightarrow\qquad
\text{$X_j F \in \langle\uP\rangle$ pour tout $j$}
$$
En effet, on a $F \in \uPsat_\delta$ et, d'après le point ii) de la
proposition \ref{MiniWiebe}, $\uPsat_\delta = \langle\uP\rangle_\delta
\oplus \bfA\nabla$ où $\nabla$ est un déterminant bezoutien de
$\uP$. L'implication résulte alors du fait que $X_j\nabla \in
\langle\uP\rangle$.

\medskip

$\blacktriangleright$
Soit l'exemple $D = (1,2,2)$, de degré critique $\delta=2$, fourni
après la proposition \ref{FormesInertieDegred}.  Alors, avec les
notations de cet exemple, la forme linéaire $\ell_{1,2}$ appartient à
$\uPsat_1$. On peut \emph{expliciter} $X_3\ell_{1,2} \in
\langle\uP\rangle$ à l'aide d'un système de Calcul Formel. Mais
$X_i\ell_{1,2} \notin \langle\uP\rangle$ pour $i=1,2$, ce qui est
réalisable à la main puisqu'il suffit de l'obtenir pour une
spécialisation de $\uP$. On choisit de spécialiser $P_2$ en
$p_2X_2^2$, $P_3$ en~$p_3X_3^2$. On se rapporte alors à l'exemple illustrant
la proposition \ref{FormesInertieDegred} où on a noté $P_1 = a_1X_1+a_2X_2+a_3X_3$.
La matrice $(\uX,\uY)$-bezoutienne rigidifiée de $\uP$ spécialisé fournit
un $(\uX,\uY)$-bezoutien très simple:
$$
\Bez(\uX,\uY) = 
\begin {vmatrix}
a_1 & 0             & 0 \\
a_2 & p_2(X_2+Y_2)  & 0 \\
a_3 & 0             & p_3(X_3+Y_3) \\
\end {vmatrix} =
a_1p_2p_3 (X_2+Y_2)(X_3+Y_3)
$$
Les $\ell_{i,j}$ sont définis par $\ell_{i,j} := a_i\Bez_j(\uX)-
a_j\Bez_i(\uX)$ où les $\Bez_i(\uX)$ s'obtiennent en voyant
$\Bez(\uX,\uY)$ dans $\bfA[\uX][\uY]$ et en écrivant sa composante
linéaire en $\uY$ sous la forme:
$$
\Bez_1(\uX) Y_1 +  \Bez_2(\uX) Y_2 +  \Bez_3(\uX) Y_3 
$$
On trouve $\Bez_1 = 0$, $\Bez_2 = a_1p_2p_3X_3$ et $\Bez_3 =
a_1p_2p_3X_2$, si bien que $\ell_{1,2} = a_1^2 p_2p_3 X_3$.
Pour montrer que $X_i\ell_{1,2} \notin \langle P_1, p_2X_2^2,
p_3X_3^2\rangle$ pour $i=1,2$, on spécialise encore via $a_1 = p_1 = p_2 := 1$. On
a ainsi $\ell_{1,2} = X_3$ et l'idéal $\langle P_1, p_2X_2^2,
p_3X_3^2\rangle$ se réduit à l'idéal $\fa$ fonction de $a_2,a_3,X_1,X_2,X_3$:
$$
\fa(a_2,a_3,X_1,X_2,X_3) = \langle X_1+a_2X_2+a_3X_3, X_2^2, X_3^2\rangle
$$
Pour voir $X_1X_3, X_2X_3 \notin \fa$, on réalise une dernière spécialisation
forçant $X_1 + a_2X_2 + a_3X_3 = 0$: 
$$
\begin {array}{c} 
\fa(a_2:=-1, a_3:=0, X_1, X_2:=X_1,X_3) = \langle X_1^2,X_3^2\rangle
\qquad \text {et l'on a} \qquad
X_1X_3 \notin \langle X_1^2,X_3^2\rangle
\\[2mm]
\fa(a_2:=0, a_3:=-1, X_1:=X_3, X_2,X_3) = \langle X_2^2,X_3^2\rangle
\qquad \text {et l'on a} \qquad
X_2X_3 \notin \langle X_2^2,X_3^2\rangle
\\
\end {array}
$$
\end {rmq}

\subsection{Régularité sur $\bfA[\protect\uX]/\protect\uPsat$ et son impact
sur la $\bfB_\delta$-régularité}  
\label{SectPsatRegBdeltaRegI}

Commençons par le fait évident: 

\begin {fact} \label{RegModuloPsat}
Pour n'importe quelle suite $\uP$, le morphisme naturel suivant est injectif :
$$
\bfA[\uX]/\uPsat  \ \lhook\joinrel\xrightarrow{\qquad} \ \prod_{j=1}^n \bfB_{x_j}
\qquad 
\hbox{où $\bfB = \bfA[\uX]/\langle \uP \rangle$}
$$
En particulier,  si $a \in \bfA$ est régulier sur chaque localisé $\bfB_{x_j}$, 
alors $a$ est régulier sur~$\bfA[\uX]/\uPsat$.
\end{fact}

\medskip

Le théorème ci-dessous s'applique en particulier, d'après la
proposition~\ref{detW1deltaQ}, au jeu circulaire $\uQ$ et au scalaire
$a = \det W_{1,\delta}(\uP)$.  L'hypothèse de généricité de $\uP$ y
est utile à deux reprises: nous avons besoin d'une part de spécialiser
certains coefficients des~$P_i$ et d'autre part de la structure des
localisés~$\bfB_{x_j}$.

On peut essayer d'en éclairer la conclusion très forte, qui consiste
en la $\bfA[\uX]/\uPsat$-régularité du scalaire~$a$. Un tel scalaire
est à fortiori un élément régulier de $\bfA/(\ElimIdeal)$. Anticipons
sur la suite en supposant le résultant $\Res(\uP)$ défini.  Dans le
contexte du théorème où $\uP$ est générique, $\Res(\uP)$ est un scalaire
régulier (on verra que c'est le cas sous la seule hypothèse que $\uP$
est une suite régulière) et générateur de l'idéal d'élimination
$\ElimIdeal$ (ceci sera l'objet de nos préoccupations ultérieures).
Bilan: $a$ est un élément régulier de $\bfA/\langle\Res(\uP)\rangle$,
et en conséquence:
$$
(\Res(\uP), a) \text{ est une suite régulière}
$$
Ce n'est pas rien! C'est ce que fournit le théorème ci-dessous grâce à une hypothèse
spécifique de spécialisation du scalaire~$a$. Et quitte à anticiper, dans le
cas où $a = \det W_{1,\delta}(\uP)$, ce déterminant est multiple de
$\omegaRes{\uP}(X^\emouton)$ et on obtient ainsi deux scalaires aux \og caractères antinomiques\fg{}
dans l'image de $\omegaRes{\uP}$ avec la propriété:
$$
\big(\omegaRes{\uP}(\nabla_\uP),\ \omegaRes{\uP}(X^\emouton) \big) \text{ est une suite régulière}
$$
Antinomiques car le premier est dans l'idéal d'élimination (c'en est même un générateur)
tandis que le second est en dehors de cet idéal en un sens trés fort
(il est régulier modulo le premier). La présence du mouton-noir est-elle remarquable?
Non, car on verra en section \ref{SectPsatRegBdeltaRegII} qu'il peut être remplacé par n'importe
quel monôme de degré $\delta$ fournissant ainsi un certain nombre de suites régulières
de longueur~$2$ dans l'image de $\omegaRes{\uP}$.

\begin{theo}[Une condition suffisante de régularité sur $\bfA\lbrack\uX\rbrack/\uPsat$
             en générique] \leavevmode
\label{PsatReg}

Soit $\uP$ le système générique sur $\bfk$ dont nous notons $\bfA = \bfk[\indetsPi]$
l'anneau des coefficients.
On suppose disposer d'un système $\uQ$, à coefficients dans $\bfk$, de même
format que celui de $\uP$, tel que $\Un \overset{\rm def.}{=}
(1 : \cdots : 1)$ soit un zéro de ce système~$\uQ$.

\medskip

Alors tout $a \in \bfA$ dont la spécialisation en $\uQ$ vaut 1
est régulier sur $\bfA[\uX]/\uPsat$ et de surcroît un scalaire régulier.
\end{theo}

\begin{proof} \leavevmode

$\rhd$ 
Gardons les notations de~\ref{StructureLocalise} et considérons un indice $1\le j \le n$.
Il existe un polynôme $a' \in \bfA'_j[\uX]$ 
tel que $X_j^e \,a \equiv a' \bmod \langle \uP \rangle$ pour un certain exposant~$e$.
Les spécialisations $\uP := \uQ$ et $\uX := \Un$ 
fournissent $a'(\uQ,\Un) = 1$ (on utilise le fait que $Q_i(\mathds 1) = 0$ et l'hypothèse $a(\uQ) = 1$).
Par conséquent 
$a' \in \bfA'_j[\uX]$ est régulier en tant que polynôme primitif par valeur.
Donc l'élément $a'/X_j^e$ de $\bfA'_j[\uX]_{X_j}$ est régulier.
Par l'isomorphisme $\bfA'_j[\uX]_{X_j} \simeq \bfB_{x_j}$, la classe
$\overline a \in \bfB_{x_j}$ qui lui correspond est également un élément régulier
de l'anneau localisé $\bfB_{x_j}$.

\noindent
Ainsi $a$ est régulier sur chaque $\bfB_{x_j}$ donc, en vertu du fait évident~\ref{RegModuloPsat},
régulier sur~$\bfA[\uX]/\uPsat$.

\medskip

$\rhd$ 
Par ailleurs, comme $a$ est primitif par valeur vu que $a(\uQ)=1$, il est régulier sur $\bfA$.
\end{proof}

\medskip

Nous relions maintenant les deux notions de régularité en montrant que
la régularité sur $\bfB_\delta$ se déduit de celle sur
$\bfA[\uX]/\uPsat$, ce qui va permettre de compléter le théorème
précédent (cf le corollaire après le lemme). Le lemme est de nature
déductive et il n'y a pas à supposer~$\uP$ générique car tout le sel
de l'affaire repose sur les épaules de Wiebe. Insistons sur le fait
que la $\bfB_\delta$-régularité est une propriété bien plus faible que
la régularité sur~$\bfA[\uX]/\uPsat$: un scalaire régulier qui est
$\bfA[\uX]/\uPsat$-régulier est régulier sur $\bfB$ tout entier! Cf la
remarque qui suit le corollaire~\ref{corBdeltaReg}.

\begin{lem}[La régularité sur $\bfA\lbrack\uX\rbrack/\uPsat$ entraîne celle sur $\bfB_\delta$]
\label{PsatRegImpliesBdeltaReg}
\leavevmode
Soit $\uP$ une suite régulière.
Tout scalaire régulier $a \in \bfA$, régulier sur $\bfA[\uX]/\uPsat$, est $\bfB_\delta$-régulier.
\end{lem}

\begin{proof} \leavevmode
  
Pour $F \in \bfA[\uX]_\delta$ tel que $a \overline F = 0$ dans $\bfB_\delta$,
on doit montrer que $\overline F = 0$. On a $aF \in \uPdelta$, a fortiori,
$aF \in \uPsat_\delta$ et donc d'après
l'hypothèse de régularité sur $\bfA[\uX]/\uPsat$, il vient $F\in \uPsat_\delta$.

\medskip

D'après le théorème de Wiebe~(cf.~\ref{MiniWiebe}), il existe $\lambda
\in \bfA$ tel que $F - \lambda \nabla \in \uPdelta$.  En multipliant
par~$a$, et en réinjectant l'hypothèse $aF \in \uPdelta$, on
obtient $a \lambda \nabla \in \uPdelta$.  Puis $a \lambda = 0$ dans
$\bfA$ (toujours d'après Wiebe).  Puisque $a$ est un scalaire
régulier, on en déduit $\lambda = 0$ donc $F \in \uPdelta$,
c'est-à-dire $\overline F = 0$.
\end{proof}

En enchaînant les deux résultats précédents, on obtient le corollaire suivant.

\begin {coro} \label{corBdeltaReg}
Dans le contexte du théorème \ref{PsatReg},  le scalaire $a$ est de plus $\bfB_\delta$-régulier.
\end {coro}

\begin{rmqs} [A propos de régularité sur $\bfB_d$, sur $\bfB$] \leavevmode

Ici nous supposons $\uP$ régulière. Nous allons tirer conséquence de deux informations
fournies par le théorème de Wiebe:
$$
\forall\ d \geqslant \delta+1, \quad 
\uPsat_{d} = \langle \uP \rangle_d 
\qquad
\text{ et }
\qquad
(\uP : \uX) = \langle \uP, \nabla \rangle  \ \hbox { a fortiori }\ 
(\uP : \uX)_{\delta - 1} = \langle \uP \rangle_{\delta - 1}
$$
$\rhd$
Contexte: $d \leqslant \delta$ et un scalaire $a \in \bfA$.
\emph{Sous la seule hypothèse} que $a$ est $\bfB_\delta$-régulier,
montrons que $a$ est $\bfB_d$-régulier.  Notons $d'$ l'entier défini
par $d+d' = \delta$.  Soit $F \in \bfA[\uX]_d$ tel que $aF \in \langle
\uP \rangle_d$.  Pour tout monôme $X^\alpha$ avec $|\alpha| = d'$, on
a (banal) $aF X^\alpha \in \uPdelta$.  On utilise alors que $a$ est
$\bfB_\delta$-régulier, ce qui donne $F X^\alpha \in \uPdelta$ et ceci
pour tout $\alpha$ tel que $|\alpha| = d'$.  On a alors $F X^\beta \in
(\uP : \uX)_{\delta - 1} = \langle \uP \rangle_{\delta - 1}$ pour tout
$|\beta| = d'-1$.  En itérant $e$ fois ce procédé, on trouve $F
X^\gamma \in \langle \uP \rangle_{\delta - e}$ pour tout $|\gamma| =
d'-e$.  En particulier, pour $e = d'$, on obtient $F \in \langle \uP
\rangle_d$, ce qu'il fallait démontrer.

\medskip
\noindent
$\rhd$
Contexte: un scalaire régulier $a \in A$, régulier sur
$\bfA[\uX]/\uPsat$. Alors $a$ est régulier sur $\bfB$ c'est-à-dire
$\bfB_d$-régulier pour tout $d$. Tout d'abord, d'après le lemme
\ref{PsatRegImpliesBdeltaReg}, $a$ est $\bfB_\delta$-régulier.  Donc
d'après le premier point de cette remarque, il est $\bfB_d$-régulier
pour $d \le \delta$.  Examinons le cas $d \geqslant \delta+1$.  Par
hypothèse, $a$ est régulier sur $(\bfA[\uX]/\uPsat)_d$.  Or
$\uPsat_{d} = \langle \uP \rangle_d$.  Ainsi, $a$ est régulier sur
$\bfB_d$.
\end{rmqs}

\subsection{En générique: $\bfA[\protect\uX]/\protect\uPsat$-régularité via
la spécialisation en le jeu circulaire}
\label{SectPsatRegBdeltaRegII}

Nous avons étudié dans le chapitre \ref{ChapJeuCirculaire} un certain nombre
de scalaires $a(\uP)$ (déterminants, évaluations de formes linéaires), définis
pour n'importe système $\uP$ de format $D$, dont la  spécialisation
en le jeu circulaire $\uQ$ de même format $D$, a pour valeur $a(\uQ) = 1$.
Nous allons en récolter les bénéfices.

\subsection*{Une liste de scalaires $(\bfA[\uX]/\uPsat)$-réguliers, a fortiori
$\bfB_\delta$-réguliers}

Voici la liste de scalaires $a(\sbullet)$ (ce sont tous des
déterminants) accompagnée de références pour lesquels nous avons
prouvé $a(\uQ) = 1$ où nous rappelons qu'ici $\uQ$ désigne le jeu
circulaire.  Il y a les déterminants $\det W_{1,\delta}$ et $\det
W_\calM$ pour $\calM \subseteq \Jex_{1,\delta}$, cf la proposition
\ref{detW1deltaQ}.

Viennent ensuite les évaluations $\omega(X^\alpha)$ où $X^\alpha$ est
de degré $\delta$, cf la proposition \ref{ProprietesJeuCirculaire}. Il
n'est pas inutile de rappeler, cf le commentaire avant la proposition
en question, qu'en le jeu circulaire, la forme $\omega_\uQ :
\bfA[\uX]_\delta \to \bfA$ est indépendante de $D$; elle ne dépend que
de $\delta$ puisque c'est l'évaluation $F \mapsto F(\Un)$ au point
$\Un = (1 : \cdots : 1)$.

Terminons par les sections ``bonus'': la section
\ref{SousSectionEnDessousDelta} (premier bonus) concerne $\det
W_\calM$ pour $\calM \subseteq \Jex_{1,d}$ et $d \le \delta$. Enfin la
section \ref{SousSectionW2Qtriangulaire} (second bonus), cf
proposition \ref{WkdQTriangSup} règle le cas de $\det W_\calM$ avec
$\calM \subseteq \Jex_{2,d}$ pour n'importe quel~$d$.

\medskip
Bien entendu, tous ces résultats associés au mécanisme de sélection $\minDiv$
s'adaptent à la $\sigma$-version tordue par $\sigma \in \fS_n$.
Nous obtenons alors le résultat principal de ce chapitre.

\medskip

\begin{theo}\label{GenPsatRegScalars}
  
Soit $\uP$ la suite générique sur $\bfk$. On dispose de la liste suivante
de scalaires réguliers de l'anneau~$\bfA$ des coefficients de $\uP$, qui sont
réguliers sur $\bfA[\uX]/\uPsat$, a fortiori $\bfB_\delta$-réguliers.
Dans cette énumération, $\sigma$ désigne une permutation de $\fS_n$.

\begin {enumerate}[1.]
\item  
Les déterminants $\det W^\sigma_{1,\delta}(\uP)$.

\item
Plus généralement, les déterminants $\det W^\sigma_{\calM}(\uP)$ pour
tout sous-module monomial $\calM \subseteq \Jex_{1,\delta}$.

\item
Les évaluations $\omega^\sigma_\uP(X^\alpha)$ pour $X^\alpha$ de degré $\delta$.

\item
Les déterminants $\det W^\sigma_{\calM}(\uP)$ pour $\calM \subseteq \Jex_{1,d}$ et $d \le\delta$.

\item
Les déterminants $\det W^\sigma_{\calM}(\uP)$ pour $\calM \subseteq \Jex_{2,d}$ et n'importe quel $d$.
\end {enumerate}  
\end{theo}

\begin{proof}
Le théorème~\ref{PsatReg} reformulé avec le jeu circulaire $\uQ$ stipule :
\begin{quote}
\it 
Soit $\uP$ la suite générique sur $\bfk$ et $a \in \bfA = \bfk[\indetsPi]$.  Si
la spécialisation de $a$ en le jeu circulaire $\uQ$ vaut~$1$,
alors $a$ est régulier sur $\bfA[\uX]/\uPsat$.
La $\bfB_\delta$-régularité du scalaire~$a$ se déduit du corollaire \ref{corBdeltaReg} du
théorème cité.
\end{quote}
\end{proof}

\medskip

En un certain sens, on peut considérer que le point~{\it 1.} du
théorème ci-dessus est une conséquence du théorème du
puits~\ref{TheoPuits}, résultat qui a permis d'obtenir la structure
triangulaire de $W_{1,\delta}(\uQ)$ (cf.~\ref{detW1deltaQ}),
conduisant directement à $\det W_{1,\delta}(\uQ) = 1$.  Pour le point
{\it 5.}, rappelons, cf.~\ref{SousSectionW2Qtriangulaire}, que
l'aspect triangulaire de $W_\calM(\uQ)$ a été obtenu de manière
directe pour tout sous-module monomial $\calM$ de $\Jex_{2,d}$ et pour
tout entier $d$ ; et que, dans ce cas, la preuve a été beaucoup plus
aisée que celle faite pour $W_{1,\delta}(\uQ)$.

\subsection*{L'égalité $\Ker\omega=\protect\uPdelta$ en générique ou
  l'injectivité de $\overline\omega: \bfB_\delta\hookrightarrow\bfA$}

Pour une suite $\uP$ quelconque, la forme linéaire $\omega$ est nulle
sur $\uPdelta$ \idest{} $\uPdelta \subseteq \Ker \omega$.  Le cas
d'égalité n'a pas toujours lieu et on peut même avoir $\omega = 0$
comme le prouve l'exemple de la suite régulière $(aY, bX, cZ^3)$, cf
la proposition~\ref{YXZ3}. En revanche, en terrain générique, il y a
bien égalité!

\medskip

Le théorème suivant reprend une partie de la
page~\pageref{BdeltaRegImpliesInjectivity} et la précédente (cf le diagramme
d'implications en introduction).

\index{forme linéaire!$\omega=\omega_\uP$ (pilotée par $\minDiv$)}%

\begin{theo} \label{omegaInjective}
Soit $\uP$ la suite générique.

\begin{enumerate}[\rm i)]
\item  
Le noyau de la forme linéaire $\omega : \bfA[\uX]_\delta \to \bfA$ est $\uPdelta$
\idest{}  
la forme $\overline \omega : \bfB_\delta \to \bfA$ est injective.

\item
Tout scalaire $a \in \bfA$ régulier est $\bfB_\delta$-régulier.
\end {enumerate}  
\end{theo}

\begin{proof} \leavevmode

\noindent i)
La preuve a été donnée en introduction (page~\pageref{BdeltaRegImpliesInjectivity})
en utilisant $\det W_{1,\delta} = \omega(X^\emouton)$. Nous la reprenons
ici pour appuyer le fait que le mouton-noir n'a pas de rôle privilégié.

\smallskip

\emph{Fixons} un monôme $X^\alpha$ de degré $\delta$.
Rappelons que $\omega$ est de Cramer: $\omega(G)F - \omega(F)G \in
\uPdelta$ pour $F, G \in \bfA[\uX]_\delta$. Pour $F \in
\bfA[\uX]_\delta$ tel que $\omega(F) = 0$, on doit montrer que $F \in
\uPdelta$. En particulier, pour $G = X^\alpha$, on a
$$
\omega(X^\alpha)F \ - \ \underbrace{\omega(F)}_{=0} X^\emouton \ \in \ \uPdelta
\qquad \text{d'où } \qquad 
\omega(X^\alpha)F \ \in \ \uPdelta
$$
Or $\omega(X^\alpha)$ est $\bfB_\delta$-régulier  (confer~\ref{GenPsatRegScalars}).
On obtient donc $F \in \uPdelta$.

\medskip
\noindent ii) On peut faire la preuve directe suivante:
pour $\overline F \in \bfB_\delta$ tel que $a\,\overline F = 0$, on doit montrer que $\overline F = 0$.
En appliquant $\overline\omega$ et en utilisant la régularité de $a$ dans $\bfA$, on
obtient $\overline\omega(\overline F) = 0$ donc $\overline F=0$. 

Il est cependant préférable de prendre du recul, de penser que le $\bfA$-module
$\bfB_\delta$ se plonge dans $\bfA$ et d'utiliser le fait suivant.
\end{proof}

\begin {fact} \leavevmode
  
\begin {enumerate}[\rm i)]
\item
Tout scalaire régulier sur un module $M$ est régulier
sur n'importe quel sous-module de $M$.  
\item
Tout scalaire régulier est régulier sur n'importe quel module libre $\bfA^r$,
à fortiori régulier sur n'importe quel sous-module d'un module libre.
\item    
Soit $\uP$ un système tel que le $\bfA$-module $\bfB_\delta$ se plonge dans $\bfA$,
ce que l'on schématise via $\bfB_\delta \hookrightarrow\bfA$. C'est le cas par exemple si
$\overline\omega : \bfB_\delta \to \bfA$ est injective.
Alors tout scalaire de $\bfA$ régulier est $\bfB_\delta$-régulier.
\end {enumerate}  
\end {fact}

\medskip

Dans la suite, nous allons mettre l'accent sur plusieurs subtilités.
La première concerne le fait que la condition \og $\uP$
générique\fg{} n'est pas nécessaire pour obtenir un plongement de $\bfB_\delta$
dans $\bfA$. La seconde
concerne le fait que l'on pourrait avoir $\omega^\sigma = 0$ pour tout $\sigma$
mais que $\bfB_\delta$ se plonge tout de même dans $\bfA$ (par~$\uomegares$!).

\subsection{$\bfB_\delta$-régularité: étude d'exemples non génériques}

Nous avons choisi le premier cas d'école de format
$D = (1,\cdots,1,e)$ de degré critique $\delta = e-1$, cf la
section~\ref{soussectionPremierCasEcole}. On en rappelle les
notations: $\uP = (L_1, \cdots, L_{n-1}, P_n)$ où les $L_j$ sont $n-1$
formes linéaires et $P_n$ un polynôme homogène de degré $e$ jouant un
rôle secondaire puisque:
$$
\langle\uP\rangle_\delta = \langle L_1, \cdots, L_{n-1}\rangle_\delta
$$
En posant $\xi_i = \det(L_1, \dots, L_{n-1}, X_i)$, nous y avons
défini $\uxi = (\xi_1, \cdots, \xi_n)$ et la forme $\evalxi :
\bfA[\uX]_\delta \to \bfA$, restriction à $\bfA[\uX]_\delta$ du
morphisme d'évaluation en $\uxi$.  La question se pose de déterminer
dans quels cas on dispose d'une part de l'égalité:
$$
\Ker\evalxi \overset{?}{=} \langle\uP\rangle_\delta
$$
et d'autre part de scalaires de $\bfA$ qui soient $\bfB_\delta$-réguliers.

\medskip

C'est un problème qui nous semble complexe et nous allons y
répondre que très partiellement  en fournissant un cas particulier
\emph {pas du tout générique} que nous jugeons pertinent. On rappelle le
principe qui a été utilisé dans la preuve du théorème
précédent~\ref{omegaInjective}.  Pour $F, G \in \bfA[\uX]_\delta$, on a
$$
\evalxi(G)F - \evalxi(F)G \in \langle L_1, \cdots, L_{n-1}\rangle_\delta
$$
Supposons $\evalxi(F) = 0$. Comment, si tout va bien, peut-on obtenir $F \in \langle
L_1, \cdots, L_{n-1}\rangle_\delta$?
On a bien entendu $\evalxi(G)F \in \langle L_1, \cdots, L_{n-1}\rangle_\delta$;
en particulier, en prenant pour $G$ un monôme $X^\alpha$ de degré $\delta$:
$$
\xi^\alpha\, F \in \langle L_1, \cdots, L_{n-1}\rangle_\delta
$$
On aimerait bien pouvoir simplifier par $\xi^\alpha$ modulo $\langle L_1, \cdots, L_{n-1}\rangle_\delta$,
ce qui en général n'est pas licite.

\subsection*{Un premier exemple}

Voici le cas particulier pour lequel nous allons pouvoir réaliser cette simplification:
des indéterminées $\bsp = (p_1, \cdots, p_{n-1})$,
$\bsq = (q_1, \cdots, q_{n-1})$ sur un anneau de base $\bfk$ et les $L_j$ donnés
par
$$
\left\{
\begin {array} {ccl}
L_1 &=& p_1X_1  + q_1X_2 \\
L_2 &=& p_2X_2  + q_2X_3 \\
    &\vdots&  \\
L_{n-1} &=& p_{n-1}X_{n-1}  + q_{n-1}X_n \\  
\end {array} 
\right.
$$
On laisse le soin au lecteur de vérifier que
$$
\xi_i = p_1\cdots p_{i-1} \overline{q}_i\cdots \overline{q}_{n-1} \qquad
\qquad \text{où} \qquad \overline {q}_j = -q_j
$$

\begin {theo} \label{BdeltaRegNotGeneric}
On garde le contexte ci-dessus et on note $T$ une indéterminée sur $\bfA := \bfk[\bsp,\bsq]$.

\begin {enumerate}[\rm i)]
\item
Le morphisme $\varphi$ d'évaluation en $T\uxi$, défini par $X_i \mapsto \xi_i\,T$,
fournit un plongement
$$
\bfA[\uX]/\langle L_1,\cdots,L_{n-1}\rangle  \quad\hookrightarrow\quad   \bfA[T]
$$  

\item
Tout élément régulier de $\bfA$ est  régulier modulo $\langle L_1, \cdots, L_{n-1}\rangle$.

\item
Pour chaque $\delta \ge 0$, en notant $\evalxi : \bfA[\uX]_\delta \to \bfA$ la
restriction du morphisme d'évaluation en $\uxi$
$$
\Ker\evalxi = \langle L_1, \cdots, L_{n-1}\rangle_\delta
$$
\end {enumerate}  
\end {theo}

La preuve figure en page~\pageref{BdeltaRegNotGenericProof} après le
lemme crucial ci-dessous.  Dans ce lemme dont l'énoncé est rendu autonome,
nous gardons le même contexte mais pour éviter des signes, nous changeons $q_j$ en $-q_j$
de sorte que:
$$
L_j = p_jX_j - q_jX_{j+1} \quad 1 \le j \le n-1, \qquad
\xi_i = p_1\cdots p_{i-1} q_i\cdots q_{n-1} \quad 1 \le i \le n
$$
En particulier:
$$
\xi_1 = q_1\cdots q_{n-1},\qquad  \xi_n = p_1\cdots p_{n-1}
$$
On vérifie directement que $L_j(\uxi) = 0$ et comme $L_j$ est homogène (de degré~1), $L_j(T\uxi) = 0$.

\medskip

Un commentaire s'impose concernant l'évaluation en $T\uxi$ versus
l'évaluation en $\uxi$.  Le noyau du morphisme d'évaluation en $\uxi$
est bien entendu $\langle X_1-\xi_1, \cdots, X_n-\xi_n\rangle$ mais
ceci ne nous intéresse pas. Le noyau qui nous concerne, bien plus
complexe à déterminer, est celui du morphisme $\varphi$ d'évaluation
en~$T\uxi$. Ce dernier noyau est un idéal \emph {homogène}, ce qui
permet ensuite de considérer son intersection avec chaque composante homogène
$\bfA[\uX]_\delta$ de $\bfA[\uX]$.  On peut tenter une comparaison avec la différence
résidant entre l'idéal d'un point affine et l'idéal d'un point
projectif. Imaginons que les $\xi_i$ soient à coefficients dans un
corps; alors l'idéal du point affine $(\xi_1, \cdots, \xi_n) \in
\bbA^n$ est $\langle X_1 - \xi_1, \cdots, X_n-\xi_n\rangle$ tandis
l'idéal du point projectif $(\xi_1 : \cdots : \xi_n) \in \bbP^{n-1}$
est $\langle \xi_iX_j - \xi_j X_i, \ 1 \le i,j \le n\rangle$.

\smallskip

Enfin, dans notre cadre, les $\xi_i$ ne sont pas à coefficients dans
un corps et les $\xi_iX_j - \xi_j X_i$, certes dans le noyau
du morphisme $\varphi$, n'engendrent pas ce noyau. On voit par exemple
qu'il va falloir \og dégraisser\fg:
$$
\xi_2X_1 - \xi_1X_2 = p_1q_2\cdots q_{n-1}\,X_1 - q_1\cdots q_{n-1}\,X_2 =
q_2 \cdots q_{n-1}\, L_1
$$

\begin {lem} \leavevmode

Soit $\bfA = \bfk[\bsp,\bsq]$ et $\varphi$ le $\bfA$-morphisme $\bfA[\uX] \to \bfA[T]$ qui
réalise $X_i \mapsto \xi_i\,T$. Alors:
$$
\Ker\varphi = \langle L_1, \cdots, L_{n-1}\rangle
$$
\end {lem}

\medskip

Voici le schéma de la preuve.
Nous allons considérer les domaines de départ et d'arrivée de $\varphi$
comme des anneaux de polynômes sur $\bfk$:
$$
\bfA[\uX] = \bfk[\bsp,\bsq,\uX], \qquad\bfA[T] = \bfk[\bsp,\bsq,T]
$$
et profiter du fait que $\varphi$ transforme un monôme (au sens $\bfk$)
en un monôme.

Notons $I$ l'idéal $\langle L_1, \cdots, L_{n-1}\rangle$, contenu dans
$\Ker\varphi$. La stratégie va consister à exhiber un supplémentaire
\emph {monomial} $S$ de $I$ tel que la restriction de $\varphi$ à $S$
soit injective i.e.  $\Ker\varphi \cap S = \{0\}$:
$$
\bfk[\bsp,\bsq,\uX] = I + S
$$
Nous avons volontairement écrit $+$ au lieu $\oplus$ car le fait que
la somme ci-dessus soit directe est conséquence de $\Ker\varphi \cap S
= \{0\}$ vu que $I \subset \Ker\varphi$.  Pour en déduire l'égalité
$\Ker\varphi = I$, on prend $F \in \Ker\varphi$ que l'on décompose en
$F = F_1 + F_2$ avec $F_1 \in I$ et $F_2 \in S$. En appliquant
$\varphi$, on obtient $\varphi(F_2) = 0$ donc $F_2=0$ et $F = F_1 \in
I$.

\begin {proof} \leavevmode

\noindent $\bullet$
Définition de $S$. On prend pour $S$ le $\bfk$-module de base les monômes
non multiples de $p_jX_j$ pour chaque $1 \le j \le n-1$. Pour montrer
que $\bfk[\bsp,\bsq,\uX] = I + S$, il suffit de montrer que tout
monôme $m$ non dans $S$ est congru modulo $I$ à un monôme dans $S$.
Considérons le \emph {plus petit} $j$ tel que $m$ soit multiple de $p_jX_j$
et supposons que $j=1$. On remplace alors $p_1X_1$ par $q_1X_2$
qui vérifie $q_1X_2 \equiv p_1X_1 \bmod I$. Dans le monôme obtenu,
on recommence si besoin ce genre d'opération fournissant un $j > 1$
(le plus petit) et on remplace $p_jX_j$ par $q_jX_{j+1}$. 
Et ainsi de suite.

Ainsi, avec $n = 4$, montrons comment obtenir, à partir du monôme $m$ (non dans $S$) à gauche ci-dessous,  
le monôme $m'$ (dans $S$) à droite, vérifiant $m' \equiv m \bmod I$:
$$
m = p_1^3 p_2^2 p_3\, X_1^2X_3, \qquad m' = p_1 q_1^2 q_2^2 q_3 X_3^2X_4
$$
Voici la suite des opérations:
$$
\xymatrix @R = 0.3cm @C = 2.5cm{
m = (p_1X_1)^2 p_1p_2^2 p_3X_3 \ar[r]^-{p_1X_1\leadsto q_1X_2} & m_1 = (q_1X_2)^2 p_1 p_2^2 p_3X_3 =
q_1^2 p_1 p_2^2 p_3 X_2^2 X_3  
\\
m_1 = (p_2X_2)^2 q_1^2 p_1p_3X_3 \ar[r]^-{p_2X_2\leadsto q_2X_3} & m_2 = (q_2X_3)^2 q_1^2 p_1 p_3X_3 =
q_2^2 q_1^2 p_1 p_3 X_3^3
\\
m_2 = (p_3X_3) q_2^2 q_1^2 p_1X_3^2 \ar[r]^-{p_3X_3\leadsto q_3X_4} & (q_3X_4) q_2^2 q_1^2 p_1 X_3^2 =
q_3 q_2^2 q_1^2 p_1 X_3^2 X_4 = m'
\\
}
$$

\noindent $\bullet$
Injectivité de la restriction de $\varphi$ à $S$. Comme des monômes distincts
sont $\bfk$-linéairement indépendants et que $\varphi$ transforme un monôme
en un monôme, il suffit de montrer que la restriction de $\varphi$ à
l'ensemble des monômes de $S$ est injective.

Déterminons l'image par $\varphi$ d'un monôme $m = \bsp^\alpha \bsq^\beta X^\gamma$
en posant:
$$
s_i = \gamma_1 + \cdots + \gamma_i, \qquad t_i = \gamma_{i+1} + \cdots + \gamma_n
\quad 1 \le i \le n
$$
Il est clair $\varphi(m)$ est de la forme $\bsp^? \bsq^? T^{s_n}$.
En utilisant $\xi_i = p_1 \cdots p_{i-1} q_i \cdots q_{n-1}$, on obtient 
les contributions de $\bsq$ et $\bsp$ dans $\varphi(m)$:
$$
q_1^{\beta_1 + s_1}\, q_2^{\beta_2 + s_2}\, \cdots\, q_{n-1}^{\beta_{n-1} + s_{n-1}},
\qquad\quad
p_1^{\alpha_1 + t_1}\, p_2^{\alpha_2 + t_2}\, \cdots\, p_{n-1}^{\alpha_{n-1} + t_{n-1}} 
$$
Considérons maintenant deux monômes $m = \bsp^\alpha \bsq^\beta X^\gamma$ et
$m' = \bsp^{\alpha'} \bsq^{\beta'} X^{\gamma'}$ \framebox [1.2\width][c]{dans $S$} tels que
$\varphi(m) = \varphi(m')$ et montrons que $m = m'$.
L'appartenance de $m$ à $S$, \idest{} $m$ non divisible par $p_jX_j$, se traduit
par $\alpha_j=0$ ou $\gamma_j=0$ pour $1 \le j \le n-1$.  Idem pour $m'$.

L'exposant en $T$ dans $\varphi(m) = \varphi(m')$ donne $s_n = s'_n$ que l'on désigne
par $s$. L'exposant en $p_1$ dans $\varphi(m) = \varphi(m')$  fournit
$$
\alpha_1 + t_1 = \alpha'_1 + t'_1
\qquad \text{\idest} \qquad
\alpha_1 + s - \gamma_1 = \alpha'_1 + s - \gamma'_1 
\qquad \text{donc} \qquad
\alpha_1 - \gamma_1 = \alpha'_1 - \gamma'_1 
$$
Mais $\alpha_1= 0$ ou $\gamma_1 = 0$; idem avec des $'$. On laisse le soin au
lecteur de vérifier que cela entraîne $\alpha_1 = \alpha'_1$ et $\gamma_1 = \gamma'_1$!

L'exposant en $p_2$ dans $\varphi(m) = \varphi(m')$  fournit
$$
\alpha_2 + t_2 = \alpha'_2 + t'_2
\qquad \text{\idest} \qquad
\alpha_2 + (s-\gamma_1) - \gamma_2 = \alpha'_2 + (s-\gamma'_1) - \gamma'_2
\qquad \text{donc} \qquad
\alpha_2 - \gamma_2 = \alpha'_2 - \gamma'_2 
$$
Mais $\alpha_2= 0$ ou $\gamma_2 = 0$; idem avec des $'$. Via le même \emph{trick},
on obtient $\alpha_2 = \alpha'_2$ et $\gamma_2 = \gamma'_2$.

\medskip
En continuant ainsi, on obtient $\alpha = \alpha'$ et $\gamma = \gamma'$. Puis
on obtient $\beta = \beta'$ en considérant les exposants en $q_j$.
Bilan: $m=m'$, ce qu'il fallait démontrer.
\end {proof}   

\medskip

Quelques commentaires à propos de la décomposition
$\bfk[\bsp,\bsq,\uX] = I \oplus S$.  Le choix de $S$ est guidé par la
notion de base de Gröbner. Munissons $\bfk[\bsp,\bsq,\uX]$ de l'ordre
monomial lexicographique pour lequel $p_1 > \cdots > p_{n-1} > q_1 >
\cdots > q_{n-1} > X_1 > \cdots > X_n$.  La forme $L_j$ admet alors
$p_jX_j$ comme monôme dominant. Le fait que les $p_jX_j$ soient
premiers deux à deux assure, lorsque $\bfk$ est un corps ou bien
$\bbZ$, que $(L_1, \cdots, L_{n-1})$ est une base de Gröbner de $I$ et
donc l'idéal initial de $I$ est $\langle p_1X_1, \cdots,
p_{n-1}X_{n-1}\rangle$.  Les monômes de $S$ sont ceux ``sous
l'escalier'' et sont dits normalisés ou standards.  Ce supplémentaire
$S$ est parfois qualifié de supplémentaire de Macaulay de $I$. En ce
qui concerne l'exemple $(m,m')$ donné dans la preuve, $m'$ est la
forme normale de $m$ relativement à l'ordre monomial retenu.

Avec un anneau de base quelconque, lorsque les monômes dominants sont
premiers deux à deux, il y a un traitement général: le lecteur pourra
consulter le problème (corrigé) 9 du chapitre X de la seconde édition
de~\cite{LombardiQuitte}, problème intitulé ``Quand les monômes dominants sont
premiers entre eux''.

\begin {proof} [Preuve du théorème \ref{BdeltaRegNotGeneric}] \label{BdeltaRegNotGenericProof} \leavevmode

 \noindent
i) Découle directement du lemme.  

\medskip
\noindent
ii) Soit $a \in \bfA$ un scalaire régulier et $F \in \bfA[\uX]$ tels que
$aF \in \langle L_1, \cdots, L_{n-1}\rangle$. Il s'agit de montrer
$F \in \langle L_1, \cdots, L_{n-1}\rangle$. On donne un coup de $\varphi$
i.e. $0 = a\varphi(F)$ et en utilisant la régularité de $a$, on obtient
$\varphi(F) = 0$ donc $F \in \langle L_1, \cdots, L_{n-1}\rangle$.

\medskip
\noindent
iii) Relire l'introduction avant le premier exemple et utiliser
le fait que chaque $\xi_i$ est régulier modulo $\langle L_1, \cdots,
L_{n-1}\rangle$.  Ou bien considérer la restriction de~$\varphi$ à
$\bfA[\uX]_\delta$ et utiliser le fait que pour $F \in
\bfA[\uX]_\delta$, on a $\varphi(F) = \evalxi(F)\,T^\delta$.
\end {proof}

\subsection*{Un second exemple}

Ici $n=3$, $p,q,r,s$ désignent des indéterminées sur $\bbZ$ et $\uP = (L_1,L_2,P_3)$
le système de format $D = (1,1,3)$ de degré critique $\delta = 2$ avec:
$$
L_1 = pX_2 + qX_3, \qquad  L_2 = rX_1 + sX_3, \qquad  P_3 = X_3^3
$$
Alors $\omega^\sigma = 0$ pour tout $\sigma \in \fS_3$. Cependant
$\evalxi : \bfA[\uX]_\delta \to \bfA$, par passage au quotient, permet
de plonger $\bfB_\delta$ dans $\bfA$. En ce qui concerne la preuve,
nous allons fournir quelques pistes, les détails étant laissés au
lecteur. A noter que le polynôme $P_3$ n'intervient pas dans la suite
(il n'intervient pas dans la composante homogène de degré $\delta$ du
complexe de Koszul de $\uP$).  Nous l'avons choisi ainsi car il permet
de montrer facilement que la suite $\uP$ est régulière et donc que son
complexe de Koszul est exact. On a
$$
\dim\rmK_{2,\delta} = 1, \qquad  \dim\rmK_{1,\delta} = \dim\rmK_{0,\delta} = 6 
$$
avec les différentielles:
$$
\Syl_\delta = \partial_{1,\delta} =
\NorthEastBordermatrix{
\Veti{X_{1}\,e_{1}} & \Veti{X_{2}\,e_{1}} & \Veti{X_{3}\,e_{1}} & \Veti{X_{1}\,e_{2}} & \Veti{X_{2}\,e_{2}} & \Veti{X_{3}\,e_{2}} & \\
. & . & . & r & . & . & \Heti{X_{1}^{2}} \\
p & . & . & . & r & . & \Heti{X_{1}X_{2}} \\
q & . & . & s & . & r & \Heti{X_{1}X_{3}} \\
. & p & . & . & . & . & \Heti{X_{2}^{2}} \\
. & q & p & . & s & . & \Heti{X_{2}X_{3}} \\
. & . & q & . & . & s & \Heti{X_{3}^{2}} = \mouton \\
}
\qquad\qquad
\partial_{2,\delta} =
\NorthEastBordermatrix{
\Veti{e_{12}} & \\
-r & \Heti{X_{1}\,e_{1}} \\
. & \Heti{X_{2}\,e_{1}} \\
-s & \Heti{X_{3}\,e_{1}} \\
. & \Heti{X_{1}\,e_{2}} \\
p & \Heti{X_{2}\,e_{2}} \\
q & \Heti{X_{3}\,e_{2}} \\
}
$$
On en tire les deux matrices $\Omega^\sigma$ suivantes:
$$
\Omega =
\EastBordermatrix{
. & . & . & . & . & \sbullet & \Heti{X_{1}^{2}} \\ 
p & . & . & r & . & \sbullet & \Heti{X_{1}X_{2}} \\ 
q & . & . & . & r & \sbullet & \Heti{X_{1}X_{3}} \\ 
. & p & . & . & . & \sbullet & \Heti{X_{2}^{2}} \\ 
. & q & p & s & . & \sbullet & \Heti{X_{2}X_{3}} \\ 
. & . & q & . & s & \sbullet & \Heti{X_{3}^{2}} \\ 
}
\qquad\qquad
\Omega^{(1,2,3)} =
\EastBordermatrix{
. & r & . & . & . & \sbullet & \Heti{X_{1}^{2}} \\ 
p & . & . & r & . & \sbullet & \Heti{X_{1}X_{2}} \\ 
q & s & . & . & r & \sbullet & \Heti{X_{1}X_{3}} \\ 
. & . & . & . & . & \sbullet & \Heti{X_{2}^{2}} \\ 
. & . & p & s & . & \sbullet & \Heti{X_{2}X_{3}} \\ 
. & . & q & . & s & \sbullet & \Heti{X_{3}^{2}} \\ 
}
$$
On laisse le soin au lecteur de vérifier que leurs 5 premières
colonnes sont solidement liées (indication: dans $\Omega$, il manque
la colonne 4 de $\Syl_\delta$ et dans $\Omega^{(1,2,3)}$, il manque la
colonne 2; puis utiliser $\partial_{2,\delta}$, qui a la bonne idée
d'avoir un 0 en ligne~2 et ligne~4). Enfin, toute $\Omega^\sigma$
est égale soit à $\Omega$ soit à $\Omega^{(1,2,3)}$. Bilan:
$\omega^\sigma = 0$ pour tout $\sigma \in \fS_3$.

\medskip

Passons à $\evalxi$:
$$
\begin {vmatrix}
.  &r  &T_1 \\  
p  &.  &T_2 \\
q  &s  &T_3 \\  
\end {vmatrix}
= \xi_1 T_1 + \xi_2 T_2 + \xi_3 T_3 
\qquad \text{avec} \qquad
\uxi \overset{\rm def.}{=} (\xi_1, \xi_2,\xi_3) = (ps, qr, -pr)
$$
Enfin, la matrice de $\evalxi$ dans base déjà utilisée des monômes de degré $2$ de $\bfA[X_1,X_2,X_3]_2$
est:
$$
\begin {bmatrix}
\xi_1^2 &\xi_1\xi_2 &\xi_1\xi_3 &\xi_2^2 &\xi_2\xi_3 &\xi_3^2
\end {bmatrix}
=
\begin {bmatrix}
p^2s^2 &pqrs &-p^2rs &q^2r^2 &-pqr^2 &p^2r^2  
\end {bmatrix}
$$
Au lecteur de montrer que $\Ker\evalxi = \langle L_1,L_2\rangle_2$.

\cleardoublepage

\section{AC$_2$: idéaux et sous-modules de Fitting, modules de \mbox{MacRae}}
\label{ChapFittingVectoriel}

L'objectif principal de ce chapitre est de définir la notion de
module de MacRae avec une attention toute particulière pour ceux de
rang $0$ ou $1$. Il nous a semblé opportun de commencer en revisitant rapidement la
notion d'idéal de Fitting, sans toutefois passer sous silence les
propriétés de base. Nous avons choisi de les appuyer sur un traitement
non habituel (Schanuel).  En fin de chapitre, nous avons fait figurer
un aperçu de la littérature consacrée au sujet ``Idéaux de Fitting''.

Nous avons tenu à faire figurer un certain nombre d'exemples en dehors
du contexte de l'élimination, car, à notre connaissance, la notion de
module de MacRae de la littérature n'est développée qu'en rang 0,
cadre dans lequel l'invariant de MacRae associé est un scalaire (à un
inversible près). Ici, un module $M$ de MacRae de rang $c$ porte un
invariant (de MacRae), qui est une forme
$c$-linéaire alternée $\vartheta \in \bigwedge^c(M)^\star$,
\emph {définie à un inversible près}.

\index{invariant de MacRae}%

\subsection{Idéaux de Fitting et modules (de présentation finie) de rang constant}
\label{sous-sectionIdeauxFitting}

\`A un module de présentation finie $M = \Coker u$ présenté par $u : E \to F$, on associe 
ses idéaux de Fitting définis à l'aide des idéaux déterminantiels de $u$:
$$
\calF_k(M) 
\ = \ \calD_{\dim F-k}(u)
$$ 
Cette définition a du sens car elle ne dépend pas de $u$ (cf. bien-fondé des idéaux 
de Fitting ci-après).

On a bien sûr les inclusions :
$$
0 = \calF_{-1}(M) 
\ \subset \ 
\calF_0(M) 
\ \subset \
\calF_1(M) 
\ \subset \
\calF_2(M) 
\ \subset \
\cdots 
$$
qui finissent par stagner sur l'anneau $\bfA$ à partir de l'indice
$q$, si $M$ est engendré par $q$ générateurs.

\label {NOTA08-Dru}%
\label {NOTA08-FkM}%
\index{idéal!de Fitting d'un module de présentation finie}%

Lorsqu'il existe $c$ tel que 
$$
\calF_{c-1}(M) = 0 
\qquad \text{ et } \qquad 
\text{$\calF_c(M)$ fidèle}
$$
un tel $c$ est unique, et on dit que $M$ est de rang $c$, ou encore
que $u$ est de rang $r = \dim F - c$.  Cette notion de rang est celle
de \og rang stable \fg{} de Northcott dans ses chapitres $3$ et $4$.
Attention, dans~\cite[définition 5.7 du chapitre II]{LombardiQuitte},
la terminologie \og rang = $r$ \fg{} pour une matrice $u$ est
différente: il s'agit du \og rang au sens fort\fg{}, ce qui signifie
$1 \in \calD_r(u)$ et $\calD_{r+1}(u) = 0$.

\index{module!de rang constant}%

\smallskip

Donnons quelques exemples. 

$\rhd$ Le module libre $\bfA^c$ est évidemment de rang $c$. 

$\rhd$ Soit $I$ un idéal de type fini. On a $\calF_0(\bfA/I) = I$. 
Ainsi, le module $\bfA/I$ est de rang $0$ si et 
seulement si $I$ est fidèle, ce qui est par exemple le cas lorsque $I$ contient un élément régulier.

$\rhd$ Tout module librement résoluble $M$ possède un rang en ce sens, 
et le rang est la caractéristique d'Euler-Poincaré de n'importe quelle résolution libre
de $M$.

\medskip

L'idéal de Fitting d'indice $0$, particulièrement important, est relié à l'annulateur :

\begin{prop}[$0$-Fitting d'un module de présentation finie, 
{cf.~\cite[IV.9.6.]{LombardiQuitte}}] \label{0FittingAnnulateur}

Soit $M$ un module de présentation finie engendré par $g$ générateurs.
Alors 
$$
\Ann(M)^g 
\ \subset \ 
\calF_0(M)
\ \subset \ 
\Ann(M)
$$
\end{prop}

\subsubsection*{Bien-fondé des idéaux de Fitting}

C'est Fitting, qui en 1936, associa à un module $M$ de présentation finie la famille des idéaux qui 
porte son nom et étudia leurs propriétés ;
ceux-ci sont définis comme idéaux déterminantiels d'une matrice de présentation finie de $M$
et la première chose à assurer est l'indépendance vis-à-vis de cette présentation.
Même si ce résultat est très souvent cité dans la littérature, 
mettre effectivement le doigt sur une preuve est plus difficile, 
et parfois source d'étonnements.
Plus loin, nous aurons besoin de ce résultat d'indépendance, 
mais dans une version enrichie, valable dans un cadre plus général.
Pour ces deux raisons, il nous semble indispensable d'assurer 
les détails et si possible avec des arguments efficaces.

Nous aurons besoin du lemme de Schanuel. 
Il en existe deux versions.
La version classique :
\begin{quote}
\it 
Soit $\pi: F \twoheadrightarrow M$ et $\pi' : F' \twoheadrightarrow M$ 
deux applications linéaires surjectives, de même image~$M$, où $F$ et $F'$ sont projectifs.
Alors il existe un isomorphisme entre $\Ker \pi \oplus F'$ et $F \oplus \Ker \pi'$.
\end{quote}
et la version améliorée (confer~\cite[V, lemme 2.7]{LombardiQuitte}) :
\begin{quote}
\it 
Soit $\pi: F \twoheadrightarrow M$ et $\pi' : F' \twoheadrightarrow M$ 
deux applications linéaires surjectives, de même image~$M$, avec les modules $F$ et $F'$ projectifs.
Alors il existe des isomorphismes réciproques $\Phi, \Phi' : F \oplus F' \to F \oplus F'$ 
faisant commuter le diagramme suivant 
$$
\xymatrix @R=2pt @M=0.5pc{
F\oplus F' \ar@<-1ex>[dd]_{\Phi} \ar@{->>}[dr]^{\pi \oplus 0}  \\
            & M \\
F \oplus F' \ar[uu]_{\Phi'} \ar@{->>}[ur]_{0 \oplus \pi'}  \\
}
$$
En particulier, $\Phi$ transforme $\Ker(\pi \oplus 0) = \Ker \pi \oplus F'$
en $\Ker(0 \oplus \pi') =  F \oplus \Ker \pi'$ ; 
idem pour $\Phi'$.
\end{quote}
\`A l'aide de Schanuel version améliorée, démontrons le résultat suivant.
\begin {prop} [Fitting from Schanuel] \label{FittingInvariance}
\leavevmode
Pour deux présentations d'un même module $M$, on~a les égalités des idéaux déterminantiels : 
$$
\vcenter {
\xymatrix @R = 0.01cm{
E \ar[r]^u & F \ar@{->>}[dr]  \\
           &                &M \\
E'\ar[r]_{u'} & F'\ar@{->>}[ur] \\
}
}
\qquad\qquad
\quad \forall\, k, \quad \calD_{\dim F-k}(u) =\calD_{\dim F' -k}(u')  
$$
\end{prop}

\begin{proof}
On pose $f = \dim F$, idem avec un prime et 
on note avec un tilde les applications prolongées par l'identité ;  
c'est-à-dire $\widetilde u = u \oplus \id_{F'}$ et idem pour $u'$.
Remarquons dès maintenant que cette opération tilde ne modifie pas 
les idéaux déterminantiels dans le sens où 
pour tout $j$, on a 
$\calD_{j+f'}(\widetilde u) = \calD_{j}(u)$. 
Matriciellemment, c'est évident :
$$
\calD_{j+f'} \left(\begin {bmatrix} U & 0\cr 0 & \Id_{f'} \end {bmatrix}\right) 
\ =\  
\calD_{j} (U) 
$$
Commençons la preuve à proprement parler.
D'après le lemme de Schanuel amélioré, on a le diagramme commutatif et exact suivant :
$$
\xymatrix @R=2pt @M=0.5pc{
E \oplus F' \ar[r]^{\widetilde u} & F \oplus F' \ar@<-1ex>[dd]_{\Phi} \ar@{->>}[dr]^{\pi \oplus 0}  \\
             &            & M \\
F \oplus E' \ar[r]_{\widetilde u'} & F \oplus F' \ar[uu]_{\Phi'} \ar@{->>}[ur]_{0\oplus\pi'}  \\
}
$$
En particulier, 
$\Phi$ transforme $\Im \widetilde u$ en $\Im \widetilde u'$,
ce qui s'écrit $\Im (\Phi \circ \widetilde u) =  \Im \widetilde u'$.

\noindent
On utilise à présent successivement ces deux résultats :
\begin{quote}
\it 
$\rhd$ Soient $u_1 : L_1 \to L$ et $u_2 : L_2 \to L$ 
deux applications linéaires entre modules libres telles que 
$\Im u_1 \subset \Im u_2$ ; 
alors pour tout $j$, 
on a l'inclusion $\calD_j(u_1) \subset \calD_ j(u_2)$. 
A fortiori, si $u_1,u_2$ ont même image, elles ont même idéaux déterminantiels.

\smallskip
$\rhd$ Soit $\phi$ et $\psi$ deux applications linéaires entres
modules libres, $\phi$ étant un isomorphisme du but de $\psi$
(i.e. $\phi \circ \psi$ est défini). Alors pour tout $j$, on a
$\calD_j(\phi\circ \psi) = \calD_j(\psi)$.
\end{quote}
Ils conduisent aux deux égalités successives suivantes :
$$
\calD_{f+f'-k}(\widetilde u') 
\ = \ 
\calD_{f+f'-k}(\Phi \circ \widetilde u) 
\ = \ 
\calD_{f+f'-k}(\widetilde u) 
$$
On termine en utilisant le fait que l'opération tilde ne change pas les idéaux déterminantiels :
$$
\calD_{f-k}(u)
\ = \ 
\calD_{f'-k}(u')
$$
\end {proof}

\begin{prop}[Modules de rang constant et localisation]
\label{FittingLocalisation} \leavevmode

\noindent
Pour un $\bfA$-module $M$ de présentation finie, les deux conditions sont équivalentes:

\begin {enumerate}[\rm i)]
\item
$M$ est de rang $c$.

\item
Il existe une famille finie $\us = (s_\ell)$ de scalaires vérifiant
$\Ann(\us) =0$ et telle que sur chaque localisé~$\bfA[1/s_\ell]$, le
module $M[1/s_\ell]$ soit libre de rang $c$.
\end {enumerate}

\smallskip

De plus, lorsque $M$ est de rang $c$, on peut imposer à~$\us$
d'être un système de générateurs de l'idéal de Fitting $\calF_c(M)$.
De manière précise, soit $u : \bfA^m \to \bfA^n$ une présentation de $M$ et $r =
n-c$ de sorte que $\calF_c(M) = \calD_r(u)$; alors on peut
prendre pour $\us$ la famille des mineurs d'ordre $r$ de $u$.
\end{prop}

\begin{proof} 
\leavevmode

Supposons i).
L'idéal $\calD_r(u)$ est fidèle et l'idéal $\calD_{r+1}(u)$ est nul.
Utilisons le lemme de la liberté (cf.~\cite[II.5.10]{LombardiQuitte}):
celui-ci affirme qu'une matrice $U \in \bbM_{n,m}(\bfR)$ vérifiant $\calD_{r+1}(U) = 0$ et
possédant un mineur d'ordre $r$ inversible est équivalente
à $\left[\begin {smallmatrix} \Id_r & 0 \\ 0 & 0\\ \end{smallmatrix}\right]$; en conséquence
$\Coker(U)$ est libre de rang $n - r$.
En l'appliquant à $\bfR = \bfA[1/s_\ell]$, où $s_\ell$ est un mineur d'ordre $r$
de $u$, on obtient que $M[1/s_\ell]$ est libre de rang~$c$.

\medskip

\noindent
Pour la réciproque, voici une remarque préalable en présence d'une famille
finie $\us = (s_\ell)$ vérifiant $\Ann(\us) = 0$. Soit $I$ un idéal de
$\bfA$. On laisse le soin au lecteur de vérifier que $I$ est nul (resp. fidèle)
si et seulement si chaque~$I[1/s_\ell]$ est nul (resp. fidèle).

Supposons ii). Utilisons que les idéaux de Fitting commutent à
n'importe quelle localisation au sens où
$\calF_k(S^{-1}M) = S^{-1}\calF_k(M)$. 
En appliquant la remarque préalable
à $I=\calF_{c+1}(M)$, on obtient $\calF_{c+1}(M) = 0$. Et en
l'appliquant à $I=\calF_c(M)$, on obtient que $\calF_c(M)$
est fidèle.
\end{proof}

\begin {prop} [Formes linéaires sur un module de rang 0, de rang 1] \leavevmode
\label {FormesLineairesRang01}
Soit $M$ un $\bfA$-module de présentation finie de rang $c$.
  
\begin{enumerate} [\rm i)]
\item 
Si $c= 0$, toute forme linéaire sur $M$ est nulle.

\item
On suppose $c=1$.
\begin{itemize}
\item[a)]
Deux formes linéaires $\alpha,\beta : M \to \bfA$ sont proportionnelles
au sens où $\alpha(m)\beta(m') = \alpha(m')\beta(m)$ pour tous $m,m' \in M$.

\item[b)] 
Soit $\alpha \in M^\star$ une forme linéaire.
Alors $\alpha$ est une base de $M^\star$ si et seulement si $\Gr(\alpha) \geqslant 2$.
\end{itemize}
\end{enumerate}
\end{prop}

\begin{proof} 
\leavevmode

i) 
Puisque $\calF_0(M) \subset \Ann(M)$, pour tout $a \in \calF_0(M)$, on a
$a\alpha(m) = \alpha(am) = 0$. Donc ${\calF_0(M)\alpha(m) = 0}$. Or 
$\calF_0(M)$ est un idéal fidèle, donc $\alpha(m) = 0$.

\smallskip
ii)a)
Utilisons la proposition précédente avec $c=1$. Il suffit, pour obtenir
l'égalité de scalaires $\alpha(m)\beta(m') = \alpha(m')\beta(m)$, de
localiser en $s_\ell$. Or sur ce localisé, $M[1/s_\ell]$
devient libre de rang 1 donc les deux formes linéaires y sont
proportionnelles.

\smallskip
ii)b) Supposons $\Gr(\alpha) \geqslant 2$. 
Soit $\beta \in M^\star$.
On a $\alpha(m)\beta(m') = \alpha(m')\beta(m)$. 
Or l'idéal de type fini $\langle \alpha(m), m\in M\rangle$ est de profondeur $\geqslant 2$ 
par hypothèse.
Donc il existe un (unique) $q \in \bfA$ tel que $\beta(m) = q \, \alpha(m)$ 
pour tout $m \in M$ i.e. $\beta =  q \alpha$. Ainsi $\alpha$ est une base de $M^\star$.

\smallskip
Réciproquement, supposons que $\alpha$ soit une base de $M^\star$. Considérons une
présentation $E \xrightarrow[]{u}  F \buildrel \pi \over \twoheadrightarrow M$. 
Il s'agit de prouver que $\Im\alpha$, égal à $\Im(\alpha\circ\pi)$,
est de profondeur $\geqslant  2$.
Dire que $\alpha$ est une base de $M^\star$ signifie que la forme linéaire 
$\alpha \circ\pi$ est
une base du sous-module constitué des formes linéaires sur $F$ nulles
sur $\Im u$. Ce dernier module n'est autre que
$\ker \transpose{u}$ où $\transpose {u} : F^\star \to E^\star$ et le
résultat est alors conséquence du lemme suivant appliqué à
$\psi = \transpose{u}$ et $x = \alpha\circ\pi$.
\end{proof}

\begin {lem}
Soit $\psi : \bfA^m \to \bfA^n$ une application linéaire. Supposons
que $\Ker\psi$ soit libre de dimension~1 de base $x \in \bfA^m$;
alors $\Gr(x) \geqslant 2$.
\end {lem}

\begin{proof}
Il est clair que $\Ann(x)$ est réduit à $0$. 
Fixons $m$ scalaires $y_1, \dots, y_m$ vérifiant $y_ix_j = y_jx_i$ pour tous $i,j$.
En vectorisant ces égalités, 
on obtient un vecteur $y \in \bfA^m$ vérifiant $x_iy = y_ix$ pour tout~$i$.
En appliquant $\psi$, il vient $x_i\psi(y) = 0$ et comme $\Ann(x) = 0$, on obtient
$\psi(y) = 0$ \idest{} $y \in \Ker \psi = \bfA x$. Ainsi, il existe
$q \in \bfA$ tel que $y = qx$, ce qu'il fallait montrer.
\end{proof}

Eu égard au contexte de l'élimination, le point ii)b) de la proposition
précédente~(\ref{FormesLineairesRang01}) mérite la reformulation qui vient.
L'application linéaire $u$ y joue le rôle de l'application de Sylvester
$\Syl_\delta(\uP)$ entre les modules $\rmK_{1,\delta}$ et $\bfA[\uX]_\delta$,
première différentielle de la composante homogène de degré $\delta$ du complexe
de Koszul de $\uP$.

\begin {prop} \label{FormeLineaireGr2}
\leavevmode
Soit une application linéaire $u : E \to F$ entre deux $\bfA$-modules
libres, que l'on suppose de rang $\dim F - 1$. Alors une
forme linéaire $\mu$ sur $F$ est une base de $\Ker\transpose{u}$
(sous-module de $F^\star$ constitué des formes linéaires nulles sur
$\Im u$) si et seulement si $\Gr(\mu) \geqslant 2$.
\end {prop}

\subsection{Module de MacRae de rang~$0$ et son invariant (de MacRae) scalaire~$\MacRae$}
\label{InvariantMacRaeRang0}

Sur $\bbZ$, il est facile d'expliquer ce qu'est un module de MacRae de
rang~$0$: c'est tout simplement un groupe abélien fini; et son
invariant de MacRae est son cardinal.  Pour assurer le lien avec le
reste de la section, il suffit, étant donné un groupe abélien $M$ de
type fini, de le \textit{présenter} \idest{}
de l'écrire (de manière non unique) $M = \bbZ^n / \Im u$ où $u : \bbZ^m \to \bbZ^n$ est une
application linéaire. Son cardinal se lit alors à l'aide d'une telle présentation de 
la manière suivante : $M = \Coker(u)$ est fini (\idest{} $M$ est de rang 0) si et seulement
si $u$ possède un mineur plein (d'ordre $n$) non nul. Auquel cas :
$$
\# M = \pgcd\big(\calD_n(u)\big)
$$
Un cas particulier  simple est celui où le module est présenté par une matrice 
\textit{carrée} $U \in \bbM_n(\bbZ)$, ce qui est toujours possible.
Le module $M = \bbZ^n / \Im U$ est fini si et seulement si $\det U \ne 0$,
auquel cas $\#M = |\det U|$.

\index{invariant de MacRae}%

\medskip
On dispose d'un résultat analogue pour un anneau principal $\bfA$:
soit $U \in \bbM_m(\bfA)$ une matrice \textit{carrée} et $M
= \bfA^m/\Im U$; alors le $\bfA$-module $M$ est de torsion (au sens
annulé par un élément régulier) si et seulement si $\det U \ne 0$
(penser à $\det U$ régulier), ou encore si 
et seulement si $M$ est de rang 0 ; dans ce cas, le produit des facteurs invariants de $M$ 
est égal à $\det U$, à un inversible près. Sur un
annneau principal, on vient de décrire la classe des modules de MacRae
de rang 0 et de leur attacher un invariant (le produit des
facteurs invariants), scalaire régulier défini à un inversible près.

\medskip
Un cas particulier: soit $u$ un endomorphisme d'un $\bfK$-espace vectoriel $E$
de dimension finie et $M$ le $\bfK[T]$-module associé; alors $M$ est de MacRae
de rang $0$ et son invariant de MacRae est le polynôme caractéristique de $u$
(à ne pas confondre avec son annulateur qui est le polynôme minimal de $u$).

\medskip
Passons maintenant à un anneau commutatif \textit{quelconque} $\bfA$ ; l'exemple
le plus simple de module de MacRae de rang $0$ que l'on puisse donner 
est celui d'un module de la forme $\bfA^m / \Im U$ où $U \in \bbM_m(\bfA)$ est une
matrice carrée dont le déterminant est régulier; son invariant de
MacRae est alors l'idéal engendré par le scalaire régulier $\det U$.

\medskip
Après cette tentative de motivation, l'objectif est de définir, sur un
anneau commutatif quelconque, la notion de modules de MacRae, en
commençant par ceux de rang~0.  Un peu de technique est nécessaire,
technique qui s'appuie essentiellement sur la profondeur $\geqslant 2$.

\subsubsection*{L'invariant de MacRae chez Northcott}

Chez Northcott \cite{NorthcottFFR}, la notion d'invariant de 
MacRae est définie en deux temps. 
Une première définition a lieu dans le chapitre 3 (The invariants of Fitting and MacRae) 
et ne couvre que le cas des modules ayant une résolution élémentaire c'est-à-dire 
une résolution finie dont chaque terme est un module élémentaire\footnote{%
Un module élémentaire est un module de présentation finie présenté par
une matrice carrée de déterminant régulier, cf. la section 3.5 (The Euler characteristic) de Northcott.}.
\`A ce stade, l'invariant de MacRae est un idéal principal fractionnaire 
(\idest{} de l'anneau total des fractions de l'anneau de base). 
La notion d'invariant de MacRae est ensuite prolongée dans le chapitre~7 
(The multiplicative structure) aux modules librement résolubles 
de rang $0$. La vision définitive est annoncée par {\sc new definition}
dans la sous-section 7.4, p. 225; pour un module~$E$ librement
résoluble de caractéristique d'Euler-Poincaré nulle, Northcott dit
``The MacRae invariant of~$E$ is defined to be the smallest principal
ideal $\langle g\rangle$ containing $\calF_0(E)$''. C'est certes vrai car 
cela dit que $\langle
g\rangle$ est l'invariant de MacRae de~$E$ si $g$ est un pgcd de
n'importe quel système générateur $\ua$ de $\calF_0(E)$; mais
il n'est pas mention de l'inégalité $\Gr(\ua/g) \geqslant 2$, pourtant capitale.



\begin{defn}\label{DefMacRae0}
Soit $M$ un module de présentation finie.
On dit que $M$ est de MacRae de rang~$0$ lorsque l'idéal $\calF_0(M)$ possède un pgcd fort.

Dans ce cas, on note $\MacRae(M)$ l'idéal principal engendré par un
tel pgcd, appelé \textit{invariant de MacRae de $M$}.  Dans la suite,
on se permettra un abus de notation en désignant parfois par
$\MacRae(M)$ n'importe quel scalaire générateur de l'idéal $\MacRae(M)$.
\end{defn}

\label {NOTA08-MacRaeScalaire}%
\index{module!de MacRae!de rang $0$}%
\index{invariant de MacRae}%

La terminologie est cohérente dans le sens où un module $M$ de MacRae de rang 0
est un module (de présentation finie) de rang 0 ! 
En effet, l'idéal $\calF_0(M)$ admet
un pgcd fort et de ce fait, par définition, c'est un idéal fidèle.

En résumé, pour un module $M$ de MacRae de rang $0$, on a :
$$
\calF_0(M) \ = \ 
\MacRae(M)\, \fb
\qquad \text{avec } \quad 
\begin{array}[t]{l}
\MacRae(M) \text{ idéal principal engendré par un élément régulier} \\
\fb \text { idéal de type fini vérifiant } \Gr(\fb) \geqslant 2
\end{array}
$$
En particulier, $\calF_0(M) \subset \MacRae(M)$.

\subsubsection*{Quelques exemples élémentaires}

$\rhd$ Soit $M = \bfA/\langle a \rangle$ avec $a \in \bfA$ régulier.
Le lecteur constatera qu'il s'agit d'un module de MacRae de rang $0$ et 
que $\MacRae(M) = \langle a \rangle$.

\medskip

$\rhd$ 
Dans la même veine, on peut considérer $M = \bfA / \fa$ où $\fa = g \fb$ est de type 
fini avec $g \in \bfA$ régulier et $\Gr(\fb) \geqslant 2$.
Comme $\calF_0(M) = \fa = g \fb$, il est clair que $M$ est 
un module de MacRae de rang $0$ et que $\MacRae(M) = \langle g \rangle$.
\medskip

$\rhd$ Prenons le $\bfA$-module 
$M = \bfA/\langle a_1 \rangle \oplus \cdots \oplus \bfA/\langle a_m \rangle$ avec 
des $a_i \in \bfA$ réguliers.
C'est un module de MacRae de rang $0$.
Il est présenté par 
$$
\xymatrix @C=4pc @M=1pc{
\bfA^m \ar[r]^-{ 
\left[
\begin{smallmatrix}
a_1 & \\
 & \ddots \\
&  & a_m \\
\end{smallmatrix}
\right]
}
& \bfA^m 
}
$$
Posons $g = \prod_i a_i$ ; c'est un élément régulier.
On a $\calF_0(M) = \langle g \rangle$ 
et $\MacRae(M) = \langle g \rangle$.

\medskip

$\rhd$ Un autre exemple : 
\label{ExempleAsurbGr2}
soit $\ub = (b_1, \dots, b_n)$ une suite de profondeur $\geqslant 2$
(par exemple, une suite régulière de longueur $n \geqslant 2$).
Alors $M = \bfA/\langle \ub \rangle$ est un module de MacRae de rang $0$.
En effet, il est présenté par 
$$
\xymatrix @C=4pc @M=1pc{
\bfA^n \ar[r]^-{ 
\left[
\begin{smallmatrix}
b_1 & \cdots & b_n 
\end{smallmatrix}
\right]
} 
& \bfA 
}
$$
L'idéal de Fitting $\calF_0(M) = \langle \ub \rangle$ peut se factoriser 
$\langle 1 \rangle \, \fb$
avec $\fb = \langle \ub \rangle$ de profondeur $\geqslant 2$.
Donc $\MacRae(M) = \langle 1 \rangle$.

\medskip

$\rhd$ Généralisons le point précédent.
Prenons $M = 
\bfA/\langle \ub^{(1)} \rangle \oplus \cdots \oplus \bfA/\langle \ub^{(p)} \rangle$
avec 
$\ub^{(1)},  \dots, \ub^{(p)}$ des suites de profondeur $\geqslant 2$.
Ce module est de MacRae de rang $0$.
Il est présenté par une matrice \og par paliers\fg{} où le palier numéro $i$ 
est la \textit{ligne} $\ub^{(i)}$. 
Ainsi, pour $p = 3$, voilà la matrice par paliers:
\begin{center}
\begin{tikzpicture}[line cap=round, line join=round,
x=1cm,y=1cm,  xscale = 1.4,  yscale = 0.6]
\clip (-0.5,0.7) rectangle (6.5,3.5);
\draw [<->] (0.,3.)-- (2.,3.);
\draw [<->] (2.,2.)-- (5.,2.);
\draw [<->]  (5.,1.)-- (6.,1.);
\draw [thick] (-0.2, 3.3)-- (-0.2, 0.7);
\draw [thick] (-0.2, 3.3)-- (-0.1, 3.3);
\draw [thick] (-0.2, 0.7)-- (-0.1, 0.7);
\draw [thick] (6.2, 3.3)-- (6.2, 0.7);
\draw [thick] (6.2-0.1, 3.3)-- (6.2, 3.3);
\draw [thick] (6.2-0.1, 0.7)-- (6.2, 0.7);
\path [fill=white] (1.,3.) circle (7pt);
\draw (1.1,3.1) node {$\underline b^{(1)}$};
\path [fill=white] (3.5, 2) circle (7pt);
\draw (3.6, 2.1) node {$\underline b^{(2)}$};
\path [fill=white]  (5.5, 1.) circle (7pt);
\draw (5.6, 1.1) node {$\underline b^{(3)}$};
\end{tikzpicture}
\end{center}
Ainsi, l'idéal des mineurs pleins de $u$, à savoir $\calF_0(M)$, 
est engendré par les $b_{i_1}^{(1)} \cdots b_{i_p}^{(p)}$ 
pour $1 \leqslant i_k \leqslant \text{longueur}\big(\ub^{(k)}\big)$. 
D'après la propriété de multiplicativité des suites de profondeur $\geqslant 2$ 
(cf.~\ref{Gr2Multiplicativity}-i), le produit des 
$p$ suites $\ub^{(1)}, \dots, \ub^{(p)}$ est de profondeur $\geqslant 2$.
Donc pour les mêmes raisons que l'exemple précédent, on a $\MacRae(M) = \langle 1 \rangle$.

\medskip

$\rhd$
Donnons un dernier exemple élaboré à partir des précédents.  Prenons $M =
\displaystyle \bigoplus_{i=1}^m \bfA/\langle a_i \rangle 
\oplus \bigoplus_{j=1}^p \bfA/\langle \ub^{(j)} \rangle$ 
avec $a_1, \dots, a_m$ des éléments réguliers et 
$\ub^{(1)},  \dots, \ub^{(p)}$ des suites de profondeur $\geqslant 2$.
Ce module est de MacRae de rang $0$.
Il est présenté par une matrice par paliers construite à partir de l'assemblage 
de la matrice diagonale des $a_1, \dots, a_m$ et de la matrice par paliers ci-dessus.

Posons $g = \prod_i a_i$.
En gardant les notations précédentes, 
on voit que l'idéal $\calF_0(M)$ des mineurs pleins de la matrice ci-dessus 
est engendré par les $g b_{i_1}^{(1)} \cdots b_{i_p}^{(p)}$. 
On a donc $\calF_0(M) = \langle g \rangle \fb$ où $\fb$ est l'idéal
$\langle\ub^{(1)}\rangle \cdots \langle\ub^{(p)}\rangle$
de profondeur $\geqslant 2$ (cf.~\ref{Gr2Multiplicativity}-i).
Ainsi $\MacRae(M) = \langle g \rangle$.

\begin {rmq} \label{MatricePaliers}
Indépendamment des histoires de modules de MacRae de rang $0$,
voici un exemple typique de matrice par paliers.
$$
\left[
\begin{array}{*{13}{c}}
a_{1}& a_{2}& a_{3}& .& .& .& .& .& .& .& .& .& . \\ 
.& .& .& b_{4}& b_{5}& b_{6}& b_{7}& .& .& .& .& .& . \\ 
.& .& .& .& .& .& .& c_{8}& c_{9}& .& .& .& . \\ 
.& .& .& .& .& .& .& .& .& d_{10}& .& .& . \\ 
.& .& .& .& .& .& .& .& .& .& e_{11}& e_{12}& . \\ 
.& .& .& .& .& .& .& .& .& .& .& .& f_{13} \\ 
\end{array}
\right]
$$
A tout palier, on associe l'idéal engendré par les éléments du palier. Par exemple,
l'idéal associé au premier palier est $\langle a_1,a_2,a_3\rangle$.
On dispose alors du résultat suivant: l'idéal des mineurs pleins d'une matrice par
paliers est le produit des idéaux associés aux paliers.
\end {rmq}

\subsubsection*{Invariant de MacRae d'un module très décomposé}

Le dernier exemple, qui engloble tous les précédents, s'énonce de la manière suivante.

\begin{prop} \label{ModuleTresDecompose}
Soit $\mathscr A$ une famille finie d'éléments réguliers $a \in \bfA$
et $\mathscr B$ une famille finie de suites $\ub$ de profondeur $\geqslant 2$.
Alors le module 
$$
\displaystyle \bigoplus_{a \in \mathscr A} \bfA/\langle a \rangle 
\ \oplus \ 
\bigoplus_{\ub \in \mathscr B} \bfA/\langle \ub \rangle
$$
est un module de MacRae de rang $0$ 
et son invariant de MacRae est $\prod\limits_{a \in \mathscr A} a$.
\end{prop}

\index{invariant de MacRae}%

\begin{proof}
Elle est contenue dans le dernier exemple ou dans la
remarque; elle consiste donc à chasser dans les mineurs d'une présentation.

Une autre preuve s'appuie sur le cas des modules $\bfA/\langle a \rangle$ et 
$\bfA/\langle \ub \rangle$ (traités dans les exemples)
et sur la multiplicativité de l'invariant de MacRae pour une somme directe 
(confer l'énoncé ci-après). Si on y regarde de près, on voit d'ailleurs
que le résultat de multiplicativité réside également en une certaine chasse
dans les mineurs, ce qui n'est guère étonnant vu la définition de l'invariant
de MacRae.
\end{proof}

\subsubsection*{Multiplicativité de l'invariant de MacRae}

Commençons par un cas particulier extrêmement simple.

\begin{prop}[Multiplicativité de l'invariant de MacRae pour une somme directe]
Soient $M'$ et $M''$ deux modules de MacRae de rang $0$. 
Alors $M' \oplus M''$ est un module de MacRae de rang~$0$ et
$$
\MacRae(M'\oplus M'') 
\ = \ 
\MacRae(M') \MacRae(M'')
$$
\end{prop}

\index{invariant de MacRae}%

\begin{proof}
Une présentation de $M' \oplus M''$ est donnée par 
$\begin{bmatrix} A' & 0 \\ 0 & A'' \end{bmatrix}$ 
où $A'$ (resp. $A''$) est une matrice de présentation de $M'$ (resp. $M''$).
En considérant l'idéal des mineurs pleins, on obtient (grâce au développement de Laplace) :
$\calF_0(M' \oplus M'') = \calF_0(M') \calF_0(M'')$.
Les idéaux $\calF_0(M')$ et $\calF_0(M'')$ admettent un pgcd fort. 
Donc, par multiplicativité 
(cf.~\ref{Gr2Multiplicativity}-ii), 
l'idéal $\calF_0(M') \calF_0(M'')$ admet un pgcd fort qui est le produit de ces deux pgcd.
\end{proof}

Poursuivons avec la multiplicativité de l'invariant de MacRae dans un cas un peu plus 
général : celui d'une suite exacte. 
Pour cela, commençons par dégager un résultat local, local ayant ici le sens de 
\og localisation en une famille de profondeur $\geqslant 2$ \fg{}.

\begin{lem}[Présentation carrée \og locale \fg{} 
d'un module de MacRae de rang~$0$]
\label{LocalSquarePresentation}

Soit $M$ un module de MacRae de rang $0$.
Il existe une famille finie $\ub = (b_J)$ 
de scalaires vérifiant $\Gr(\ub) \geqslant 2$ telle que 
sur chaque localisé $\bfA[\frac{1}{b_J}]$, le module $M$ soit présenté
par une matrice carrée de déterminant générateur de $\MacRae(M)$
(idéal invariant de MacRae de $M$). 
\end{lem}

\begin{proof}
Soit $u : \bfA^m \to \bfA^r$ une présentation de $M$.
Notons $(e_j)_j$ la base canonique de $\bfA^m$ et pour une partie $J$
de cardinal $r$ de $\{1..m\}$, $\Delta_J$ le mineur
d'ordre $r$ correspondant de $u$. Par définition, en désignant par $g$
un générateur de $\MacRae(M)$, on a $\Delta_J = g b_J$ avec
$\Gr(\ub) \geqslant 2$.

Nous allons utiliser la formule de Cramer. On  en rappelle un
exemple typique pour 4 vecteurs $v_1, v_2, v_3$ et $v \in \bfA^3$ 
$$
\det(v_1,v_2,v_3)v 
\ = \ 
\det(v,v_2,v_3)v_1 + \det(v_1,v,v_3)v_2 + \det(v_1,v_2,v)v_3
$$
Pour une partie $J \subset \{1..m\}$ de cardinal $r$ et un indice
quelconque $k \in \{1..m\}$, la formule de Cramer appliquée aux $r+1$
vecteurs $u(e_j)_{j \in J}$ et $u(e_k)$ de $\bfA^r$ s'écrit
$$
\Delta_J \, u(e_k) \ = \ 
\sum_{j \in J} \Delta_{J\setminus j \vee k}\, u(e_j)
$$
Puisque chaque $\Delta_{J'}$ est multiple de $g$ et que $g$ est
régulier, on peut simplifier par $g$ pour obtenir :
$$
b_J \, u(e_k) 
\ \in \ 
\sum_{j \in J} \bfA \, u(e_j)
$$
En localisant en $b_J$, on en déduit que $u(e_k)$ est combinaison
linéaire des $u(e_j)$ pour $j \in J$.  On obtient ainsi que \emph {sur le
localisé} $\bfA[1/b_J]$, le module $M$ est présenté par une matrice
carrée, à savoir par la restriction de~$u$ à $\bigoplus_{j \in
J} \bfA[1/b_J] e_j$, module libre de dimension $r$. Et son déterminant est
$\Delta_J = g b_J$, générateur de $\MacRae(M) = \langle g \rangle$
sur le localisé.
\end{proof}

\begin {theo} [Multiplicativité de l'invariant de MacRae dans une suite exacte]\leavevmode
\label{MacRaeMultiplicativity}

\noindent
Soit $0 \to M' \to M \to M'' \to 0$ une suite exacte avec $M'$ et $M''$ des 
modules de MacRae de rang $0$. 
Alors $M$ est un module de MacRae de rang $0$ et on a $\MacRae(M) = \MacRae(M')\MacRae(M'')$.
\end {theo}

\index{théorème!de multiplicativité!de l'invariant de MacRae}%

\begin{proof}

Soit, pour le module $M'$, une famille $\ua = (a_i)$ comme dans le lemme
précédent; idem pour le module~$M''$ avec une famille $\ub = (b_j)$.
On a $\Gr(\ua\,\ub) \geqslant 2$ d'après \ref{Gr2Multiplicativity}.
Localisons en $a_ib_j$. D'après le lemme, sur $\bfA[1/(a_ib_j)]$,
le module~$M'$ est présenté par une matrice \textit{carrée} $C'$ ayant
un déterminant engendrant~$\MacRae(M')$.  Idem pour~$M''$.
Par conséquent, d'après \cite[prop. 4.2, chap IV]{LombardiQuitte}, le
module~$M$ est présenté par une matrice \textit{carrée} du type
$\left[\begin{smallmatrix} C'&*\\ 0 &C''\end{smallmatrix}\right]$.
\emph {Sur ce localisé}, on a l'égalité $\calF_0(M) = \MacRae(M') \MacRae(M'')$
puisque ces deux idéaux sont engendrés par $\det(C')\det(C'')$.
D'après~\ref{IdealPrincipalLocalGlobal}, on en déduit
que $\MacRae(M')\MacRae(M'')$ est un pgcd fort de $\calF_0(M)$.
Ainsi, $M$ est de MacRae de rang $0$, d'invariant de MacRae $\MacRae(M')\MacRae(M'')$.
\end {proof}

\subsection{Sous-module déterminantiel $\DVect_r(u) \subset \BW^{\dim F-r}(F)^\star$ de $u:E\to F$}
\label{DefDVect}

\index{module!déterminantiel $\DVect_r(u)$}
Soit $u : E \to F$ une application linéaire entre deux $\bfA$-modules libres.
A tout entier $r$, on va attacher un sous-module
$\DVect_r(u) \subset \bigwedge^c(F)^\star$ où $c = \dim F - r$, 
le cas vraiment \og utile\fg{} étant celui où $u$ est de rang~$r$. 

Mais pour l'instant, au niveau des généralités, nous ne faisons aucune
hypothèse particulière et désignons par $r$, $c$ deux entiers
complémentaires à $\dim F$, \idest{} $r+c = \dim F$.  La lettre $r$
(pour rang) sera plutôt attachée à $u$, tandis que le symbole $c$ sera
attaché au module $\Coker u$, celui-ci étant décrété de rang $c$
lorsque $u$ est de rang $r$. Une précision: lorsque $\Coker u$ est
librement résoluble, $c$ est également la caractéristique
d'Euler-Poincaré de n'importe quelle résolution libre de $\Coker u$.

\smallskip

On désigne par $\sharp$ un isomorphisme déterminantal entre
$\bigwedge^r(F)$ et $\bigwedge^c(F)^\star$ ; par exemple celui défini
par $\Theta \mapsto \Theta^\sharp = [\Theta \wedge \sbullet]_{\bff}$
où $\bff$ est une orientation de $F$.
Le sous-module de $\bigwedge^c(F)^\star$ en question est:
$$
\boxed{
\DVect_r(u)
\ = \ 
\bigl(\Im \textstyle \bigwedge^r(u)\bigr)^\sharp
}
$$
Par définition, l'idéal déterminantiel $\calD_r(u)$ s'obtient par évaluation de $\DVect_r(u)$ sur $F$
$$
\calD_r(u) \ = \ \scp{\DVect_r(u)}{F}   \buildrel {\rm def} \over =
\sum_{\alpha \in \DVect_r(u)} \Im \alpha
$$
Ce sous-module $\DVect_r(u)$ est indépendant de l'isomorphisme déterminantal
$\sharp$ choisi et donc ne dépend ni de la position de $\sbullet$ dans
le produit extérieur, ni de l'orientation choisie sur $F$.  On en obtient
un système fini de générateurs en fixant une base $(e_j)$ de $E$ et en
considérant les formes $c$-linéaires alternées $\bfw_J^\sharp$ sur~$F$,
avec $\#J = r$ où $\bfw_J = \bigwedge^r(u)(e_J)$ ; ces formes
$c$-linéaires alternées sont donc du type :
$$
\begin{array}{rcl}
F^c & \longrightarrow & \bfA \\
(y_1, \dots, y_c) & \longmapsto &
\big[u(e_1) \wedge \cdots \wedge u(e_r) \wedge y_1 \wedge \cdots \wedge y_c\big]_\bff
\end{array}
$$
Deux cas vont être particulièrement importants: $c = 0$ et $c = 1$.

\label {NOTA08-DVect}%

\medskip

\subsubsection*{Le cas $c=0$ \idest{} $r = \dim F$}

Rien de vraiment nouveau car
le module $\DVect_r(u)$ s'identifie à l'idéal déterminantiel $\calD_r(u)$.
En effet, $\DVect_r(u)$ est le sous-module de $\bfA^\star$ engendré 
par les formes linéaires multiplication $\delta\,\id_\bfA$ où $\delta$ parcourt
la famille des mineurs pleins de $u$, famille génératrice de
l'idéal $\calD_r(u)$.

\smallskip
Rappelons que ces mineurs~$\delta$ vérifient 
$$
\forall\, y \in F, \quad \delta y \in \Im u
$$
En effet, $\delta$ est le déterminant d'une sous-matrice $A$ carrée de taille $r$ de $u$ et
la formule de Cramer s'écrit $A \widetilde A = \delta\,\Id_r$
de sorte que $\delta y$ est combinaison linéaire des colonnes de $A$, donc appartient à $\Im u$.

L'inclusion $\calD_r(u) F \subset \Im u$ ci-dessus peut s'écrire en terme de
transporteur: $\calD_r(u) \subset (\Im u : F)$, qui  est un
visage de l'inclusion bien connue~\ref{0FittingAnnulateur} :
$$
\calF_0(M) \ \subset \ \Ann(M) 
\qquad \text{ où $M = \Coker u$}
$$

\subsection{Sous-module déterminantiel (suite): cas particulier $\DVect_{\dim F-1}(u)\subset F^\star$}
\label {SectionDetCoRang1}

\centerline{\fbox{\parbox{0.90\linewidth}{Dans toute cette section,
$u:E \to F$ est une application linéaire (entre deux modules libres de rang fini),
$r=\dim F-1$ et le sous-module déterminantiel étudié 
est $\DVect_r(u) \subset F^\star$.}}
}

\medskip

Une manière simple et visuelle
\footnote{Nous n'avons pas trouvé de notation adéquate pour désigner cette forme linéaire,
ce qui conduira dans la suite à l'utilisation d'une périphrase du type ``forme linéaire construite
à partir de $r$ colonnes de $u$''. De plus, se limiter à~$r$ colonnes de~$u$ est
restrictif puisqu'on peut remplacer ces $r$ colonnes par $r$
vecteurs de l'image de $u$.}
de construire une forme linéaire appartenant à $\DVect_r(u)$ consiste
à considérer $r$ colonnes quelconques de $u$ et à leurs associer la forme
linéaire obtenue en bordant ces~$r$ colonnes par une \og colonne joker
dans $F$\fg{}, cf. l'exemple ci-dessous et la
section~\ref{soussection-d-egal-delta} dans laquelle nous avons
qualifié une telle forme linéaire de ``forme déterminantale pure''.
Alors le sous-module de $F^\star$ engendré par ces formes
déterminantales pures est exactement le module $\DVect_r(u)$.

Par exemple, pour 
$u =  \begin{bmatrix}  a&b&c \\ a'&b'&c'\end{bmatrix} : \bfA^3 \to \bfA^2$,
le module $\DVect_r(u)$ est engendré par les formes linéaires 
$\bfA^2 \to \bfA$ suivantes 
$$
\begin{bmatrix} x\\ x'\end{bmatrix} \ \longmapsto \ 
\begin{vmatrix} 
a & x \\ a'& x'
\end{vmatrix}
\qquad 
\qquad
\begin{bmatrix} x\\ x'\end{bmatrix} \ \longmapsto \ 
\begin{vmatrix} 
b & x \\ b'& x'
\end{vmatrix}
\qquad 
\qquad
\begin{bmatrix} x\\ x'\end{bmatrix} \ \longmapsto \ 
\begin{vmatrix} 
c & x \\ c'& x'
\end{vmatrix}
$$

\subsubsection*{Forme linéaire des mineurs signés attachée au cas particulier $\dim F = 1+\dim E$}

Ici l'application linéaire est donc de la forme $u : \bfA^{r} \to
\bfA^{r+1}$.  \`A une telle application linéaire est attachée \og la
\fg{} forme linéaire $\mu : \bfA^{r+1} \to \bfA$ de ses mineurs
signés.  La terminologie est un peu abusive car cette forme linéaire
n'est définie qu'à un inversible près, selon des choix de bases.  On a
alors $\DVect_r(u) = \bfA \mu$.

Par exemple, pour $u : \bfA^2 \to \bfA^3$ de matrice
$
\begin{bmatrix} 
a&b \\ a'&b' \\ a'' & b'' 
\end{bmatrix}
$, le module $\DVect_r(u)$ est monogène et engendré par la forme linéaire
$\mu : \bfA^3 \to \bfA$ 
$$
\begin{bmatrix} x\\ x'\\ x''\end{bmatrix} \ \longmapsto \ 
\begin{vmatrix} 
a&b & x \\ a'&b'& x' \\ a'' & b'' & x''
\end{vmatrix}
$$
Il faut noter que l'information fournie par $\mu$ est beaucoup plus
précise que celle donnée par la collection des $r+1$ mineurs d'ordre
$r$ de $u$, définis au signe près et considérés de manière
indépendante: cette collection de scalaires ne permet pas, sans plus
de renseignements, de reconstruire la forme linéaire~$\mu$.  En
revanche, à partir de $\mu$, on récupère les mineurs d'ordre $r$ de $u$
puisque ce sont les composantes de $\mu$ et de ce fait l'idéal déterminantiel
$\calD_r(u) = \Im\mu$.

\subsubsection*{Un épaulement structurel de la phrase
\og border $r$ colonnes de $u$ par une colonne joker\fg?}

Considérer $r$ colonnes de $u$, donc dans $F$ module d'arrivée de $u :
E \to F$, peut être trompeur car cela masque le rôle du module~$E$,
domaine de départ de $u$. Ici, nous allons procéder autrement et
mettre en avant le choix de $r$ vecteurs d'une base du module~$E$ (dont
l'image par $u$ constitue une famille de $r$ colonnes de $u$). Ceci va
permettre l'obtention d'une famille de formes linéaires
déterminantales en \og position concordante\fg.  C'est le choix d'un
isomorphisme déterminantal $\sharp_\bff$ sur~$F$ qui va permettre de
le réaliser. Rappelons que la définition de $\sharp_\bff : F \simeq
\bigwedge^r(F^\star)$ passe par le choix d'une base ordonnée $\bff$
de~$F$ et d'une position de colonne-argument.  Ainsi en dernière
position, l'image de $y \in F$ par~$\sharp_\bff$ est, dans
l'identification \emph {canonique} $\bigwedge^r(F^\star) \buildrel
{\rm can.} \over \simeq \bigwedge^r(F)^\star$, la forme linéaire
$$
\sharp_\bff(y) \quad = \qquad
\begin{array}[t]{rcl}
{\textstyle \bigwedge^r(F)} & \longrightarrow & \bfA \\ 
y_1 \wedge \cdots \wedge y_r 
& \longmapsto 
& [y_1 \wedge \cdots \wedge y_r \wedge y]_\bff 
\end{array}
$$

\begin {prop} \leavevmode

Soit $u : E \to F$ une application linéaire entre deux modules libres
de rang fini et $r = \dim F - 1$. On fixe une base ordonnée
$\bfe$ de $E$ ainsi qu'un isomorphisme déterminantal $\sharp_\bff :
F \simeq \bigwedge^r(F^\star)$.  Pour $\bfc \subset \bfe$ de cardinal
$r$, on note $\bfe^\star_\bfc \in \bigwedge^r(E^\star)$ le produit
extérieur sur $\bfc$ des vecteurs de la base duale de $\bfe$:
$$
\bfe^\star_\bfc = c_1^\star  \wedge \cdots \wedge  c_r^\star   \quad\qquad
\bfc = (c_1, \dots, c_r)
$$
On désigne par $\mu_\bfc : F \to \bfA$ la forme déterminantale définie par le choix des $r$ colonnes
de $u$ indexées par $\bfc$ et par la position de la colonne-argument de l'isomorphisme déterminantal 
$\sharp_\bff$ fixé.

Alors la puissance $r$-ème extérieure $\bigwedge^{r}(\transpose{u}) : \bigwedge^r (F^\star) \to \bigwedge^r (E^\star)$
de la transposée de $u$, s'identifie, via l'isomorphisme déterminantal, à l'application 
$$
\begin{array}{rcl}
F  & \longrightarrow &  {\textstyle  \bigwedge^r(E^\star) } \\[0.2cm]
y & \longmapsto &
\displaystyle 
\sum_{\substack{\bfc\subset\bfe \\ \#\bfc=r}} \mu_\bfc(y)\, \bfe^\star_\bfc \qquad
\end{array}
$$
On peut schématiser le résultat de la manière suivante. Pour
$\bfc \subset \bfe$ de cardinal $r$, dessinons de manière verticale
la forme coordonnée sur $\bfe_\bfc^\star$. Alors, par composition, on
obtient la forme déterminantale~$\mu_\bfc$ schématisée de manière
diagonale en pointillés.
$$
\xymatrix @R = 1.5cm @C = 2cm {
F\ar@{..>}[drr]|{\textstyle\mu_\bfc}\ar[r]^{\sharp_\bff} & \bigwedge^{r}(F^\star) \ar[r]^-{\bigwedge^{r}(\transpose{u})}
           & \bigwedge^{r}(E^\star) \ar[d]|{\text{coordonnée sur } \bfe_\bfc^\star}  \\
           &&\bfA \\
}                                              
$$
\end {prop}

\begin {proof} \leavevmode
On va manier les identifications \emph {canoniques} comme des égalités:
$$
\textstyle 
\bigwedge^r(E^\star) =  \bigwedge^r(E)^\star, \qquad
\bigwedge^r(F^\star) =  \bigwedge^r(F)^\star, \qquad
\bigwedge^{r}(\transpose{u}) = \transpose{\Bigl(\bigwedge^{r}u\Bigr)}, \qquad
c_1^\star  \wedge \cdots \wedge  c_r^\star   =
(c_1  \wedge \cdots \wedge  c_r)^\star  
$$
Supposons la colonne-argument en dernière position et notons
$\bfc = (c_1, \dots, c_r)$. Soit $y \in F$. Son image par $\sharp_{\bff}$ est
la forme linéaire $\alpha$ sur $\bigwedge^r(F)$ définie par
$$
\alpha : y_1\wedge\cdots\wedge y_r  \mapsto [y_1\wedge\cdots\wedge y_r \wedge y]_\bff
$$
L'image de $\alpha$ par $\transpose{\Bigl(\bigwedge^{r}u\Bigr)}$ est
la forme linéaire $\beta$ sur $\bigwedge^r(E)$ définie par
$$
\beta : x_1\wedge\cdots\wedge x_r  \mapsto [u(x_1)\wedge\cdots\wedge u(x_r) \wedge y]_\bff
$$Modulo les identifications, la coordonnée de $\beta$ sur $\bfe_\bfc^\star$ est
$$
\beta(c_1  \wedge \cdots \wedge  c_r) =
[u(c_1)\wedge\cdots\wedge u(c_r) \wedge y]_\bff
$$
Ce qui est exactement $\mu_\bfc(y)$, ce qu'il fallait prouver.
\end {proof}

\subsubsection{Propriétés des formes linéaires appartenant à $\DVect_r(u)$}

\begin{prop}[Formes linéaires de Cramer] 
\label{Cramer-de-u}


Les formes linéaires $\mu \in \DVect_r(u) \subset F^\star$ vérifient
$$
\forall\, y, y' \in F, \quad 
\mu(y)y' - \mu(y') y \in \Im u
$$
Nous appellerons \og formes linéaires de Cramer \fg{} (relativement à $u$)
de telles formes linéaires. 
\end{prop}

%
%
%
%

\begin{prop}[Passage au quotient des formes linéaires déterminantales]
\label{PassageQuotientFormeLineaire}

Tout en gardant le contexte de la section, on ajoute l'hypothèse
$\calD_{\dim F}(u) = 0$ i.e. $\calD_{r+1}(u) = 0$. Et on désigne
par~$M$ le conoyau $\Coker(u)$.

\begin {enumerate} [\rm i)]
\item
Chaque forme linéaire du sous-module $\DVect_r(u) \subset F^\star$
s'annule sur $\Im u$.  Par conséquent, une telle forme linéaire passe
au quotient à travers $F \twoheadrightarrow M$ et fournit une forme
linéaire sur $M$.

\item
On suppose disposer d'une factorisation du sous-module $\DVect_r(u)
\subset F^\star$ de la forme $\mu\,\fb$ où $\fb$ est un idéal fidèle
et $\mu \in F^\star$ (on ne suppose pas $\mu$ dans $\DVect_r(u)$).
Alors la forme linéaire $\mu$ sur $F$ passe au
quotient pour fournir une forme linéaire $\overline\mu$ sur $M$
ayant même idéal image que~$\mu$.
\end {enumerate}
\end{prop}

\begin{proof} \leavevmode

i)  
Soit $\alpha$ la forme linéaire sur $F$ définie par $\alpha(y) =
[u(x_1) \wedge \cdots \wedge u(x_r) \wedge y]$ où $x_1, \dots, x_r \in
E$; alors $\alpha$ est nulle sur $\Im u$ puisque $\calD_{r+1}(u) =
0$. Idem avec une forme linéaire du type $y \mapsto
[\bigwedge^r(u).\bfx \wedge y]$ pour n'importe quel $\bfx \in
\bigwedge^r(E)$.

\medskip

ii) Soit $y \in \Im u$. Pour $b \in \fb$, comme $b\mu \in \DVect_r(u)$, on a $b\mu(y) = 0$. Ceci ayant lieu pour
tout $b\in \fb$, la fidélité de~$\fb$ conduit à $\mu(y) = 0$. Bilan: $\mu$ est nulle sur $\Im u$.
\end{proof}

\begin{rmq}
Dans le contexte du point ii), la forme linéaire $\mu$ n'a aucune
raison de vérifier la propriété de Cramer.  Prenons par exemple $u =
\begin{bmatrix} 
a \\ b 
\end{bmatrix}
: \bfA \to \bfA^2$. 
On a 
$$
\Im u = \bfA(a e_1 + be_2)
\qquad \text{ et } \qquad 
\DVect_1(u) = \bfA (ae_2^\star - b e_1^\star)
$$
Soit $d$ un élément régulier divisant  $a$ et $b$. Notons $\mu  = a'e_2^\star -b' e_1^\star$
où $a = da'$ et $b = db'$. Alors
$$
ae_2^\star - b e_1^\star \ = \ d \mu 
$$
mais $\mu(e_2) e_1 - \mu(e_1) e_2 = a' e_1 + b'e_2$ n'a aucune raison d'appartenir à $\Im u$.
\end{rmq}

\subsubsection{Lorsque $\DVect_r(u)$ possède un pgcd fort dans $F^\star$}

\label{DVectVERSUScalD}

Supposons que le sous-module $\DVect_r(u)$ de $F^\star$ admette un
pgcd fort $\mu \in F^\star$ au sens de la
définition~\ref{PgcdFortModule}. Ce qui signifie que $\mu$ divise
chaque forme de $\DVect_r(u)$ et que l'idéal (de type fini) $\fb$ défini par
$$
\DVect_r(u) \ = \mu\,\fb
\qquad \text {a la propriété} \qquad
\Gr(\fb) \geqslant 2
$$
L'évaluation sur $F$ de cette égalité fournit $\calD_r(u) = \Im\mu\,\fb$, 
a fortiori l'inclusion $\calD_r(u) \subset \Im\mu$.
On voit ainsi que l'idéal cofacteur $\fb$ 
qui intervient dans la factorisation du sous-module $\DVect_r(u)$ de $F^\star$
apparaît également dans la factorisation de l'idéal $\calD_r(u)$ de $\bfA$.

\subsection{Sous-module de Fitting $\FittVect_c(M)\subset\bigwedge^c(M)^\star$ d'un module $M$ de rang $c$}

L'objectif est d'associer à~$M$ un sous-module
$\FittVect_c(M) \subset \bigwedge^c(M)^\star$ à l'aide d'une
présentation $u : E \to F$ de $M$ et du module $\DVect_r(u)$
précédemment défini.  Tout le travail consiste à montrer que c'est une
définition intrinsèque, indépendante de la présentation
choisie.  On peut se permettre d'affaiblir l'hypothèse $M$ de rang $c$
en supposant seulement $M$ module de présentation finie
vérifiant $\calF_{c-1}(M) = 0$ (ce qui est moins contraignant: la
différence étant au niveau du module $\calF_c(M)$ dont on ne suppose
pas qu'il est fidèle).

\medskip

Pour $c = 0$, on peut régler cela rapidement en s'appuyant sur la notion d'idéal de Fitting.
En effet, $\DVect_{\dim F}(u)$ est défini \textit{directement} à partir de $\calD_{\dim F}(u)$ 
(puisqu'une forme $0$-linéaire de $\DVect_{\dim F}(u)$ est la multiplication par un élément de $\calD_{\dim F}(u)$).
Ce dernier idéal ne dépend pas de $u$, c'est l'idéal $\calF_0(M)$.
Par conséquent, $\DVect_{\dim F}(u)$ ne dépend pas de $u$ !

\medskip

Le cas $c=1$ est plus difficile, car il ne se déduit pas directement
de l'existence de $\calF_1(M)$ !  Le sous-module
$\FittVect_1(M) \subset M^\star$ que l'on va construire est plus fin
que l'idéal de Fitting $\calF_1(M)$ : deux modules peuvent avoir même
$\calF_1$ sans avoir même $\FittVect_1$.  Prenons par exemple $M
= \Coker u$ où $u = \begin{bmatrix} a\\ b\end{bmatrix} : \bfA \to \bfA^2$ avec
$\Gr(a,b) \geqslant 1$. Alors $\calF_1(M) = \langle a,b\rangle$ tandis que
$\FittVect_1(M)$ est engendré par la forme linéaire $be_1^\star -
ae_2^\star$ vue sur $M$; on peut multiplier $a$ et $b$ par des inversibles,
ce qui ne change pas $\calF_1(M)$ mais pourrait modifier radicalement
$\FittVect_1(M)$.  De plus, on peut retrouver $\calF_1(M)$ à partir de
$\FittVect_1(M)$ par évaluation sur~$M$; mais le contraire n'est pas
possible !

\medskip

Commentons l'hypothèse $\calF_0(M) = 0$ faite ci-après.  Cette égalité
est équivalente à $\calD_{\dim F}(u) = 0$ et, d'après le
résultat~\ref{PassageQuotientFormeLineaire} de passage au quotient,
permet d'interpréter une forme linéaire sur $F$ appartenant à
$\DVect_r(u)$ comme une forme linéaire de $\bigwedge^{1}(M)^\star =
M^\star$.

\begin{theo}[Le sous-module de Fitting $\FittVect_1(M) \subset M^\star$]
\label{1FittingVectoriel}

\leavevmode

Soit $M$ un module de présentation finie vérifiant $\calF_0(M) = 0$. 

\begin{enumerate}[\rm i)]
\item 
Soient $u : E \to F$, $u' : E' \to F'$ deux présentations de $M$ 
et $r = \dim F - 1$, $r'= \dim F' - 1$.
Alors les deux sous-modules $\DVect_r(u)$ et $\DVect_{r'}(u')$ vus dans $M^\star$ sont égaux. 

De manière plus formelle, dans le schéma:
$$
\vcenter 
{\xymatrix @R=2pt{
E\ar[r]^u    & F\ar@{->>}[dr]^{\pi}  \\
             &            &M \\
E'\ar[r]_{u'} & F'\ar@{->>}[ur]_{\pi'}  \\
}}
\hspace{3cm}
\vcenter 
{\xymatrix @R=2pt{
\llap{$\DVect_r(u) \ \subset \ $} F^\star \ar@{<-}[dr]^-{\transpose \pi} & \\
  & M^\star \\
\llap{$\DVect_{r'}(u') \ \subset \ $} F'^\star  \ar@{<-}[ur]_-{\transpose \pi'} & \\
}}
$$
on a l'égalité des images réciproques :
$$
(\transpose {\pi})^{-1} \big(\DVect_r(u)\big) 
\ = \ 
(\transpose {\pi'})^{-1}\big(\DVect_{r'}(u')\big)
$$
Ce sous-module de $M^\star$, indépendant de la présentation, est noté $\FittVect_1(M)$.

\item 
Les formes linéaires $\alpha \in \FittVect_1(M)$ possèdent la propriété qualifiée \og de Cramer \fg:
$$
\forall\, m, m' \in M, \quad 
\alpha(m)m' = \alpha(m') m
$$

\item
L'idéal de Fitting $\calF_1(M)$ s'obtient par évaluation de $\FittVect_1(M)$ sur $M$:
$$
\calF_1(M) = \scp{\FittVect_1(M)}{M} \buildrel {\rm def} \over = \sum_{\alpha \in\FittVect_1(M) } \Im\alpha 
$$
\item
Si $M$ est de rang $1$, alors le module $\FittVect_1(M)$ est fidèle.
\end{enumerate}
\end{theo}

\label {NOTA08-FittVect1}%
%
%

\begin{proof} \leavevmode

i)  
La stratégie de la preuve est similaire à celle
de~\ref{FittingInvariance}, qui fonde les \textit{idéaux} de Fitting.
Mais la technique est un peu plus sophistiquée et pourrait effrayer le
lecteur.  Nous avons choisi d'en déporter en annexe la preuve
(cf.~page~\pageref{Preuve1FittingVectorielInvariance}).

\medskip
ii) 
Il s'agit de relire \ref{Cramer-de-u} affirmant que les formes 
linéaires ${\mu \in \DVect_r(u) \subset F^\star}$ ont la propriété
$$
\forall\, y, y' \in F, \quad 
\mu(y)y' - \mu(y') y \in \Im u
$$

\medskip
iii)
Découle de $\displaystyle \sum\limits_{\alpha \in \DVect_r(u)} \Im \alpha = \calD_r(u)$.

\medskip
iv)
De iii) on déduit $\Ann\,\FittVect_1(M) \subset \Ann\,\calF_1(M)$
prouvant que $\FittVect_1(M)$ est fidèle si $\calF_1(M)$~l'est.
\end{proof}

Ci-après, nous illustrons la première étape de la preuve (reléguée en
annexe) dans un cas extrêmement particulier.  Cette première étape
consiste à contrôler ce qui se passe après une modification
élémentaire de la présentation.
Imaginons que $M$ soit présenté par $u : \bfA^4 \to \bfA^3$; il l'est
également par $u \oplus \Id : \bfA^5 \to \bfA^4$
$$
u = 
\begin{bmatrix}
a_1&b_1&c_1&d_1 \cr
a_2&b_2&c_2&d_2\cr
a_3&b_3&c_3&d_3\cr
\end{bmatrix}
\qquad\qquad
u \oplus \Id = 
\begin{bmatrix}
a_1&b_1&c_1&d_1&0\cr
a_2&b_2&c_2&d_2&0\cr
a_3&b_3&c_3&d_3&0\cr
0  &0  &0  & 0&1\cr
\end {bmatrix}
$$
Le théorème fournit une correspondance biunivoque canonique entre les formes
linéaires de $\DVect_2(u)$ et celles de $\DVect_3(u \oplus \Id)$.
%
Ainsi, les deux formes linéaires se correspondent:
$$
\begin{bmatrix} x_1\\ x_2\\ x_3\end{bmatrix} \ \longmapsto \ 
\begin{vmatrix}
a_1&b_1&x_1 \\
a_2&b_2&x_2 \\
a_3&b_3&x_3 \\ 
\end{vmatrix}
\qquad\qquad
\begin{bmatrix} x_1\\ x_2\\ x_3\\y\end{bmatrix} \ \longmapsto \ 
\begin{vmatrix}
a_1&b_1 & 0 & x_1 \\
a_2&b_2&  0 & x_2 \\
a_3&b_3 & 0 & x_3 \\
0 & 0 & 1 & y 
\end{vmatrix}
$$
Dans la preuve, la notion de correspondance sera précisée grâce au formalisme mis en place.

\subsubsection*{Fondement du sous-module de Fitting d'un module de rang $c$}

Ce lemme est la généralisation du résultat~\ref{PassageQuotientFormeLineaire} 
énoncé pour $c=1$.

\begin{lem}[Passage au quotient des formes $c$-linéaires alternées]
\label{PassageQuotientFormeAlternee}
\leavevmode

Soit $u : E \to F$ une application linéaire entre deux modules libres $E,F$ de rang fini.
Supposons $\bigwedge^{r+1}(u) = 0$ et notons $c = \dim F - r$, $M = \Coker(u)$.

\begin{enumerate}[\rm i)]
\item 
Toute forme linéaire du sous-module $\DVect_r(u) \subset \bigwedge^c(F)^\star$
passe au quotient à travers $\bigwedge^c(F) \twoheadrightarrow \bigwedge^c(M)$ 
et fournit une forme linéaire sur $\bigwedge^c(M)$.

\item
Supposons $\DVect_r(u)$ factorisé sous la forme $\DVect_r(u) = \mu\,\fb$ 
où $\fb$ est un idéal fidèle de $\bfA$ et $\mu \in \bigwedge^c(F)^\star$
($\mu$ est rarement dans $\DVect_r(u)$).
Alors  la forme linéaire $\mu$ passe au quotient
et fournit une forme linéaire $\overline \mu \in \bigwedge^c(M)^\star$.
 
Et par construction, l'idéal $\Im\mu$ est égal à l'idéal $\Im\overline \mu$.
 
\end{enumerate}
\end{lem}

\begin{proof} \leavevmode

i) Partons d'une forme linéaire $\bigwedge^c(F) \to \bfA$ s'écrivant
$\bfw^\sharp$ avec $\bfw \in \Im\bigwedge^r(u)$.  Montrons qu'on peut
la définir sur $\bigwedge^c(M)$, c'est-à-dire montrons qu'elle est
nulle sur le noyau de la surjection
$\bigwedge^c(F) \twoheadrightarrow \bigwedge^c(M)$.  Nous allons
utiliser le résultat d'algèbre linéaire suivant
(confer~\cite[proposition 4.9, chapitre IV]{LombardiQuitte}):
\begin{center}
\begin{quote}
\it 
Soit une suite exacte 
$\xymatrix @R=2pt{E\ar[r]^u & F\ar[r]^{\pi}  & M \ar[r] & 0 }$ de modules \textrm{quelconques}.

Alors pour tout $k$, le module $\Ker\big(\bigwedge^k \pi\big)$ est engendré par les 
$u(x) \wedge y_1 \wedge \dots \wedge y_{k-1}$ 
lorsque $x$ parcourt $E$ et les $y_i$ parcourent $F$.
Plus formellement, en définissant ${\widehat u}$ par  
$$
{\widehat u} : 
\begin{array}[t]{rcl}
E \otimes \bigwedge^{k-1}(F) & \longrightarrow & \bigwedge^{k}(F) \\
x \otimes \bfy & \longmapsto & u(x) \wedge \bfy
\end{array}
$$
la suite ci-dessous est exacte :
$$
\xymatrix @R=2pt{
E \otimes \bigwedge^{k-1}(F)\ar[r]^-{\widehat u} & \bigwedge^{k}(F) \ar[r] & \bigwedge^{k}(M) \ar[r] & 0
}
$$
Signalons que, pour $k=1$, le résultat est immédiat.
\end{quote}
\end{center}
\'Evaluons la forme linéaire $\bfw^\sharp$ en un générateur du noyau 
de la surjection $\bigwedge^c(F) \twoheadrightarrow \bigwedge^c(M)$ :
$$
\bfw^\sharp\big(u(x) \wedge y_1 \wedge \dots \wedge y_{k-1}\big) 
\ = \ 
[\bfw \wedge u(x) \wedge y_1 \wedge \dots \wedge y_{k-1}] 
$$
Par hypothèse $\bigwedge^{r+1}(u) = 0$, donc, puisque
$\bfw \in \Im\bigwedge^r(u)$, on a $\bfw \wedge u(x) = 0$ et l'évaluation
ci-dessus est nulle.  Ainsi, la forme linéaire $\bfw^\sharp$ passe au
quotient et définit une forme linéaire $\bigwedge^c(M) \to \bfA$.

\medskip

ii)
Soit $G$ un sous-module quelconque de $\bigwedge^c(F)$.  Puisque $\fb$
est un idéal fidèle, la condition \og $\mu$ s'annule sur~$G$ \fg{} est
équivalente à \og pour tout $b \in \fb$, la forme linéaire $b \mu$
s'annule sur $G$ \fg.  En l'appliquant à $G
= \Ker\big(\bigwedge^c \pi\big)$, on obtient le résultat annoncé.
\end{proof}

\bigskip

Soit $M$ un module de présentation finie tel que $\calF_{c-1}(M) = 0$,
typiquement un module de rang $c$.
De manière analogue aux cas $c = 0$ et $c=1$, on peut, à l'aide d'une présentation
de $M$, définir un sous-module de $\bigwedge^c(M)^\star$
noté~$\FittVect_c(M)$. Ce sous-module est décrit dans le résultat ci-dessous
que nous énonçons sans preuve.

Vu notre cadre d'utilisation où nous prendrons pour $M$ le
$\bfA$-module $\bfB_d = \bfA[\uX]_d/\langle\uP\rangle_d$, nous jugeons
que l'indépendance de la présentation n'est pas fondamentale puisque
$\bfB_d$ est naturellement présenté par l'application de Sylvester
$\Syl_d(\uP) : \rmK_{1,d} \to \rmK_{0,d}$.  Lorsque $\uP$ est
régulière, $\bfB_d$ est de rang $\chi_d = \dim \bfA[\uX]_d/\langle
X_1^{d_1}, \ldots, X_n^{d_n}\rangle_d$ et nous attribuerons à $\bfB_d$
ce qu'en principe, on devrait attribuer à la présentation naturelle de
$\bfB_d$. Nous nous permettrons ainsi d'écrire:
$$
\FittVect_{\chi_d}(\bfB_d) \subset \BW^{\chi_d}(\bfB_d)^\star
$$

\begin{theo}[Le sous-module de Fitting $\FittVect_c(M) \subset \BW^c(M)^\star$]
\label{cFittingVectoriel}

Soit $M$ un module de présentation finie vérifiant $\calF_{c-1}(M) =
0$ et $u : E \to F$, $u' : E' \to F'$ deux présentations
de $M$. On note $r=\dim F-c$ et $r'=\dim F'- c$.

\begin{enumerate}[\rm i)]
\item 
Le sous-module de $\bigwedge^c(M)^\star$ obtenu en passant
$\DVect_r(u)$ au quotient à travers
$\bigwedge^c(F) \twoheadrightarrow \bigwedge^c(M)$ est égal à son
analogue construit à partir de $(E',F',u')$.  De manière plus
formelle:
$$
\vcenter {
\xymatrix @R=2pt{
E\ar[r]^u    & F\ar@{->>}[dr]^{\pi}  \\
             &            &M \\
E'\ar[r]_{u'} & F'\ar@{->>}[ur]_{\pi'}  \\
}}
\hspace{3cm}
\vcenter  {
\xymatrix @R=2pt{
\llap{$\DVect_r(u)\ \subset\ $} \bigwedge^c(F)^\star\ar@{<-}[dr]^-{\transpose\,{\bigwedge^{\!c} \!\pi}}
\\
&\bigwedge^c(M)^\star
\\
\llap{$\DVect_{r'}(u')\ \subset\ $}\bigwedge^c(F')^\star\ar@{<-}[ur]_-{\transpose\,{\bigwedge^{\!c}\!\pi'}}
\\
}}
$$
on a l'égalité des images réciproques :
$$
\Big(\transpose\,{\bigwedge^{\!c} \!\pi}\Big)^{-1} \big(\DVect_r(u)\big) 
\ = \ 
\Big(\transpose\, {\bigwedge^{\!c} \!\pi'}\Big)^{-1}\big(\DVect_{r'}(u')\big)
$$
Ce sous-module de $\bigwedge^c(M)^\star$ est noté $\FittVect_c(M)$.

\item
On retrouve l'idéal de Fitting $\calF_c(M)$ par évaluation de $\FittVect_c(M)$ sur $M$:
$$
\calF_c(M) = \scp{\FittVect_c(M)}{M} \buildrel {\rm def} \over = \sum_{\alpha \in\FittVect_c(M) } \Im\alpha 
$$

\item 
Si $M$ est de rang $c$, alors le module $\FittVect_c(M)$ est fidèle.
\end{enumerate}
\end{theo}

\label {NOTA08-FittVectc}%
%
%

\subsection{Module de MacRae de rang $c$ et son invariant $\MacRaeVect=\bfA\vartheta\subset\BW^c(M)^\star$}
\label{MacRaeModule}

Les modules de MacRae que nous souhaitons définir ici forment une
classe particulière de modules de présentation finie possédant un rang,
classe contenant comme sous-classe celle des modules librement
résolubles. Nous pouvons même affirmer, dans le cadre de
l'élimination, que la plupart des modules (voire tous) utilisés ici 
sont librement résolubles. Nous avons cependant tenu à
dégager une des propriétés qui nous semble la plus importante et qui
n'utilise qu'une partie infime de ``la'' résolution, à savoir la
différentielle en degré homologique 1, c'est-à-dire une présentation du module.

\begin{defn} \label{DefModuleMacRae}

Un module de MacRae de rang $c$ est un module $M$ (de présentation finie) de rang $c$
tel que le sous-module de Fitting $\FittVect_c(M) \subset \BW^c(M)^\star$ admette un pgcd fort dans
$\BW^c(M)^\star$.

\smallskip

Par définition, un tel pgcd fort, défini à un inversible près, est une forme $c$-linéaire alternée sur $M$
sans torsion, qui divise chaque forme $c$-linéaire de $\FittVect_c(M)$;
de plus, en faisant un choix $\vartheta \in \bigwedge^c(M)^\star$ d'un pgcd,
l'idéal cofacteur $\fb$ de $\vartheta$ dans $\FittVect_c(M)$ est de profondeur
au moins~2:
$$
\FittVect_c(M) =  \vartheta\,\fb
\quad \text{avec} \quad
\Gr(\fb) \geqslant 2
$$
Le sous-module (libre de rang $1$) de $\bigwedge^c(M)^\star$ engendré par $\vartheta$ est
appelé l'invariant de MacRae de~$M$ et noté $\MacRaeVect(M)$.
\end{defn}

\label {NOTA08-MacRaeVect}%
\index{module!de MacRae!de rang $c$}%
\index{invariant de MacRae}%

Dans le cas particulier $c=0$, la forme $c$-linéaire  $\vartheta$  est 
du type $g\,\id_{\bfA}$ où $g$ est régulier. 
En identifiant $\bigwedge^0(M)^\star$ à $\bfA$, 
on retrouve alors la définition classique
de l'invariant de MacRae (cf.~\ref{DefMacRae0}).

\subsubsection*{Des modules de MacRae de rang $1$ spéciaux: les modules 1-élémentaires}

Dans la section \ref{InvariantMacRaeRang0}, on a donné comme exemple introductif de module
de MacRae de rang~0 celui d'un module présenté par une matrice carrée de déterminant régulier
et on a montré en~\ref{LocalSquarePresentation} que localement (au sens d'un système
de localisations de profondeur $\ge 2$) tout module de MacRae de rang 0
est de ce type. En ce qui concerne les modules de MacRae de rang 1, on dispose
d'une notion analogue.

\begin {defn}

Un module $M$ est dit 1-élémentaire s'il admet une présentation $u :
\bfA^r \to \bfA^{r+1}$ dont l'idéal $\calD_r(u)$ des mineurs d'ordre
$r$ est fidèle. Ou encore, en utilisant la théorème de McCoy relatif à
l'injectivité d'une matrice (cf énoncé au début de la
section~\ref{sectionHB}) s'il admet une résolution libre du type $0
\to \bfA^r \buildrel u \over \longrightarrow \bfA^{r+1}
\buildrel\pi\over\twoheadrightarrow M$.

\end {defn}

Dans le contexte de la définition, le module $\DVect_r(u)$ est monogène 
engendré par $\mu$, forme linéaire des mineurs signés de $u$. 
L'image de $\mu$ est l'idéal $\calD_r(u)$. Comme cet idéal est fidèle, 
la forme linéaire $\mu$ est sans torsion et 
est donc une base de $\DVect_r(u)$. 
En désignant par $M$ le module $\Coker(u)$, on a alors l'égalité 
$$
\FittVect_1(M) \ = \ \bfA\,\overline{\mu}
$$
faisant de $M$ un module de MacRae de rang $1$. Résumons :

\begin{fact}
Un module $1$-élémentaire $M$ est un module de MacRae de rang 1.  En
désignant par $\vartheta \in M^\star$ un générateur de MacRae, pour
toute présentation $1$-élémentaire $0 \to \bfA^r \buildrel u \over
\longrightarrow \bfA^{r+1} \buildrel\pi\over\twoheadrightarrow M$,
alors $\vartheta \circ \pi : \bfA^{r+1} \to \bfA$ est la forme
linéaire des mineurs signés de $u$.
\end{fact}

\medskip

En un certain sens, les modules 1-élémentaires constituent les briques locales
pour les modules de MacRae de rang~1. L'adjectif local est à prendre au sens
\og localisation en une suite de profondeur $\geqslant 2$ \fg{} et le résultat
précis est le suivant.

\begin{theo}[Présentation 1-élémentaire locale d'un module de MacRae de rang~$1$]
\label{1ElementaryPresentation} \leavevmode

Soit $M$ un module de MacRae de rang $1$. Il existe une famille finie
$\us= (s)$ de scalaires vérifiant $\Gr(\us) \geqslant 2$ telle que sur
chaque localisé $\bfA_s = \bfA[\frac{1}{s}]$, le module $M_s$ soit
1-élémentaire.
\end{theo}

On s'inspire très fortement de la preuve du résultat pour les modules de rang $0$ 
qui sont localement $0$-élémentaires, cf.~\ref{LocalSquarePresentation}.

\begin{proof}
Soit $u : \bfA^m \to \bfA^{r+1}$ une présentation de $M$.
Notons $(e_j)$ la base canonique de $\bfA^m$ 
et remontons l'invariant de MacRae de $M$
en une forme linéaire $\mu$ sur $\bfA^{r+1}$.

Dans la suite, $J$ désigne une partie de cardinal $r$ de $\{1..m\}$ 
et $\alpha_J$ la forme déterminantale pure sur $\bfA^{r+1}$ 
construite sur les $r$ colonnes de $u$ indexées par $J$
(cf. la section \ref{SectionDetCoRang1}).
Par exemple, pour $J = \{1..r\}$, 
on a $\alpha_J = [u(e_1) \wedge \cdots \wedge u(e_r) \wedge \sbullet]$.
Par définition, l'image de la forme linéaire $\alpha_J$ est l'idéal 
des mineurs d'ordre~$r$ de~$u$ qui s'appuient sur les colonnes  indexées par $J$.

Par définition de l'invariant de MacRae de $M$, la forme linéaire
$\mu$ est sans torsion (son image est un idéal fidèle), et on dispose
d'égalités du type $\alpha_J = \mu b_J$ où la famille $\ub = (b_J)$
est de profondeur $\geqslant 2$.

Une remarque utile en fin de la preuve : 
sur le localisé en $b_J$, l'idéal image de la forme linéaire~$\alpha_J$ 
est égal à l'idéal image de $\mu$, qui est fidèle.

Nous allons utiliser la formule de Cramer. On en rappelle un
exemple typique pour 4 vecteurs $v_1$, $v_2$, $f$, $v$ de $\bfA^3$ 
$$
\det(v_1,v_2, f)v 
\ = \ 
\det(v,v_2,f)v_1 + \det(v_1,v,f)v_2 + \det(v_1,v_2,v)f
$$
Fixons $k$ dans $\{1..m\}$. La formule de Cramer appliquée aux 
vecteurs $u(e_j)_{j \in J}$, $f$ et $u(e_k)$ de $\bfA^{r+1}$ s'écrit
$$
\alpha_J(f) \, u(e_k) \ = \ 
\sum_{j \in J} \pm \alpha_{J_j}(f) \, u(e_j)
\ + \ 
\alpha_J\big(u(e_k)\big) f
\qquad 
\text{en notant $J_j = J \setminus j \vee k$}
$$
Les mineurs d'ordre $r+1$ de $u$ étant nuls, la forme déterminantale
$\alpha_J$ est nulle sur $\Im u$ (cf. le premier point de la
proposition~\ref{PassageQuotientFormeLineaire}), en particulier
$\alpha_J\big(u(e_k)\big) = 0$.  Dans l'égalité ci-dessus, en tenant
compte de cette information et en remplaçant chaque $\alpha_{J'}(f)$
par $\mu(f)b_{J'}$, on obtient l'égalité dans $\bfA^{r+1}$ :
$$
\forall\, f \in \bfA^{r+1}, \quad 
\mu(f) \, b_J \, u(e_k) 
\ = \ 
\sum_{j \in J} \pm \mu(f) \, b_{J_j} \, u(e_j)
$$
Comme $\mu$ est sans torsion, on en déduit:
$$
b_J \, u(e_k) 
\ = \ 
\sum_{j \in J} \pm b_{J_j} \, u(e_j)
$$
En localisant en $b_J$, l'égalité ci-dessus prouve que $u(e_k)$ est combinaison
linéaire des $\big(u(e_j)\big)_{j\in J}$.  On a ainsi obtenu, sur le
localisé $\bfA[1/b_J]$, que le module $M$ est présenté par une matrice $r \times (r+1)$, 
à savoir par la restriction de~$u$ à $\bigoplus_{j \in
J} \bfA[1/b_J] e_j$, module libre de dimension $r$. 

Pour montrer que le module $M$ est $1$-élémentaire sur le localisé
$\bfA[1/b_J]$, il reste à vérifier que l'idéal engendré par les mineurs
d'ordre $r$ de cette présentation est fidèle.  Par construction, ces
mineurs sont les mineurs d'ordre $r$ de $u$ de colonnes
indexées par $J$.  Or on a fait remarquer en début de preuve que, sur le
localisé, cet idéal est l'image de $\mu$, fidèle par
hypothèse.

\smallskip
Bilan: le module $M$ est $1$-élémentaire sur $\bfA[1/b_J]$.
\end{proof}

\subsubsection*{Des modules de MacRae de rang $1$ du type 
$\bfA \oplus M_0$ avec $M_0$ de MacRae de rang $0$}

Nous fournissons ici quelques exemples de modules $M$ de MacRae de rang $1$ 
du type $\bfA \oplus M_0$ avec $M_0$ de MacRae de rang $0$.
Dans chacun des cas, nous vérifions que le module $M$ est de présentation finie,
que $\calF_0(M) = 0$, puis que le sous-module $\FittVect_1(M)\subset M^\star$ admet 
une forme linéaire pgcd fort.

\medskip

$\rhd$ Prenons le $\bfA$-module 
$M = \bfA \oplus M_0$ 
où $M_0 = \bfA/\langle a_1 \rangle \oplus \cdots \oplus \bfA/\langle a_m \rangle $ avec 
des $a_i \in \bfA$ réguliers.
Il est présenté par l'application linéaire:
$$
\xymatrix @C=4pc @M=1pc{
u : \quad \bfA^m \ar[r]^-{ 
\left[
\begin{smallmatrix}
0 & \cdots & 0 \\
a_1 & \\
 & \ddots \\
&  & a_m \\
\end{smallmatrix}
\right]
}
& \bfA \oplus \bfA^m 
}
$$
L'idéal $\calD_{m+1}(u)$ est nul et l'idéal $\calD_m(u)$ est fidèle, 
car contient le produit $a_1a_2 \cdots a_m$.
Donc $u$ est de rang $r := m$.
D'autre part, le module $\DVect_r(u)$ est engendré par la forme linéaire 
$\mu : \bfA \oplus \bfA^m \to \bfA$:
$$
\begin {bmatrix} x_0 \\ x_1 \\ \vdots \\ x_m\\ \end {bmatrix}  \ \longmapsto\ 
\left|
\begin{matrix}
0 & \cdots & 0 & x_0 \\
a_1 &  &  & x_1 \\
 & \ddots & & \vdots \\
&  & a_m & x_m \\
\end{matrix}
\right|
\ = \ 
\pm \, g\, x_0 
\qquad \text{ où $g = \prod_i a_i$}
$$
Autrement dit, en termes du module $M$, on a $\calF_0(M) = 0$ 
et le module $\FittVect_1(M)$ est engendré par la forme linéaire 
$\mu$ passée au quotient sur $M$, notée $\overline \mu$.
Ainsi, le module $M$ est de MacRae de rang~$1$ 
et $\MacRaeVect(M)$ est engendré par $\overline \mu$.

\bigskip

$\rhd$ Considérons le $\bfA$-module $M = \bfA \oplus M_0$ 
où $M_0 = \bfA/\langle \ub \rangle $ avec $\ub = (b_1, \dots, b_n)$ de profondeur $\geqslant 2$.
Il est présenté par 
$$
\xymatrix @C=4pc @M=1pc{
u : \quad 
\bfA^n \ar[r]^-{ 
\left[
\begin{smallmatrix}
0 & \cdots & 0 \\
b_1 & \cdots & b_n 
\end{smallmatrix}
\right ]
}
& \bfA \oplus \bfA^1
}
$$
L'idéal $\calD_{2}(u)$ est nul et $\calD_1(u)$ est fidèle.
Donc $u$ est de rang $r = 1$.
Et d'autre part, le module $\DVect_r(u)$ est engendré par les formes linéaires 
$\mu_i : \bfA \oplus \bfA^1  \to \bfA$ définies par 
$$
\begin {bmatrix} x\\ y\\ \end {bmatrix} \ \longmapsto\ 
\left|
\begin{matrix}
0 & x \\
b_i & y
\end{matrix}
\right| = -b_ix
$$
Comme $\Gr(\ub) \geqslant 2$, ces formes linéaires possèdent un pgcd
fort, à savoir $(x,y) \mapsto x$.  En termes du module~$M$, on a
$\calF_0(M) = 0$ et le module $\FittVect_1(M)$ est engendré par les
$b_i\vartheta$ où $\vartheta : M = \bfA \oplus M_0 \rightarrow \bfA$
est la première projection. Comme ces formes $b_i\vartheta$ ont $\vartheta$
comme pgcd fort, le module $M$ est de MacRae de rang~$1$ et
$\MacRaeVect(M)$ est engendré par $\vartheta$.  On notera l'équivalence 
$\vartheta \in \FittVect_1(M) \iff 1 \in \langle\ub\rangle$.


\medskip

Voici un énoncé qui généralise les deux exemples précédents.

\begin{prop} \label{AoplusM0}
Soit $M_0$ un module de MacRae de rang $0$. On se permet un abus de
notation en désignant par $\MacRae(M_0)$ un générateur de son idéal de
MacRae.
Alors le module $M = \bfA \oplus M_0$ est de MacRae de rang $1$ et
son invariant de MacRae $\MacRaeVect(M)$ est (engendré par) la forme linéaire
$$
\begin{array}{rcl}
\bfA \oplus M_0 & \longrightarrow & \bfA \\
x \oplus m & \longmapsto & \MacRae(M_0) x \\
\end{array}
$$
\end{prop}

\begin{proof}
Soit $u_0:\bfA^s \to \bfA^r$ une présentation de~$M_0$.
Alors $M$ est présenté par $u : \bfA^s \to \bfA \oplus \bfA^r$ 
$$
u = 
\begin{bmatrix}
0 & \cdots & 0 \\
 & & \\
 & u_0 & \\
 & & \\
\end{bmatrix}
$$
D'une part, l'idéal déterminantiel $\calD_{r+1}(u)$ est nul.  D'autre
part, pour $J$ partie de cardinal $r$ de $\{1..s\}$, désignons par
$\det_J(u_0)$ le mineur plein de~$u_0$ s'appuyant sur les colonnes
indexées par $J$; alors le module~$\DVect_r(u)$ est engendré par les formes
linéaires $\mu_J : \bfA \oplus \bfA^{r} \to \bfA$ définies par $x
\oplus y \mapsto \det_J(u_0)\, x$.
Le pgcd fort de ces formes linéaires est $x \oplus y \mapsto \MacRae(M_0) \, x$
puisque le pgcd fort des mineurs pleins de $u_0$ est, par définition, $\MacRae(M_0)$.
En termes du module~$M$, cela signifie que $\calF_0(M) = 0$ 
et que le module $\FittVect_1(M)$ est engendré par les formes linéaires
$\det_J(u_0)\vartheta$ où
$\vartheta : M = \bfA \oplus M_0 \rightarrow \bfA$ est $x \oplus m \mapsto \MacRae(M_0)x$.
Ainsi, le module $M$ est de MacRae de rang~$1$ 
et $\MacRaeVect(M)$ est engendré par $\vartheta$.


\end{proof}

\subsection{Quotient d'un module de MacRae de rang 1 par un vecteur sans torsion}
\label{sectionRank1to0}

\noindent
Cette section a été très longtemps intitulée ``Du rang 1 au rang 0'' au motif
suivant: à partir de la donnée d'un module de rang 1, on
étudie son quotient par un vecteur sans torsion, module qui \mbox{est de rang 0}.

\medskip

Rappelons les ingrédients du cadre de l'élimination.  On dispose
d'un système $\uP$ de degré critique $\delta$ et de l'application de
Sylvester $\Syl_\delta$ de $\uP$ en degré $\delta$, d'image $\langle\uP\rangle_\delta$
$$
\Syl_\delta :
\bfA[\uX]_{d_1-\delta} \times \cdots \times \bfA[\uX]_{d_n-\delta}  \to \bfA[\uX]_\delta,
\qquad
(U_1, \dots, U_n) \longmapsto \sum_{i=1}^n U_iP_i
$$
Au chapitre~\ref{ObjetsSuiteP}, il a été établi que $\overline\nabla \in \bfB_\delta
= \bfA[\uX]_\delta/\langle\uP\rangle_\delta$ est indépendant du
bezoutien $\nabla$ de $\uP$ choisi. Et en~\ref{ResolutionQuotients}, 
nous avons énoncé, lorsque $\uP$  est une suite régulière, que les $\bfA$-modules
$\bfB_\delta$ et $\bfB'_\delta = \bfB_\delta / \bfA \overline{\nabla}$
sont librement résolubles, respectivement de rang $1$ et de rang $0$.
On rappelle à cette occasion que c'est la théorie des résolutions
libres (cf le chapitre~\ref{ChapStructureMultiplicative}) qui permet d'affirmer qu'un module librement
résoluble est un module (de présentation finie) de rang constant
égal à sa caractéristique d'Euler-Poincaré.

\medskip

De manière à y voir plus clair, nous allons mettre provisoirement de
côté ce contexte de l'élimination et étudier une situation en un
certain sens plus simple (en tout cas moins encombrée).  L'objectif
est, pour un module $M$ de rang 1, de relier ses invariants à ceux du
module $M' = M/\bfA m_0$ de rang $0$, où $m_0 \in M$ est sans torsion.
Pour faciliter la lecture de ce qui suit, il est bon d'avoir
en tête le cadre de l'élimination et la correspondance suivante:
$$
M \leftrightarrow \bfB_\delta, 
\quad 
m_0 \leftrightarrow \overline \nabla, 
\quad
M' \leftrightarrow \bfB'_\delta, 
\quad
u \leftrightarrow \Syl_\delta, 
\quad 
F \leftrightarrow \bfA[\uX]_\delta, 
\quad 
r \leftrightarrow s_\delta = \dim\bfA[\uX]_\delta - 1,
\quad 
\Im u \leftrightarrow \langle \uP \rangle_\delta
$$
Mais pour l'instant, $M$ désigne un module de présentation finie et $m_0 \in M$ est quelconque.
Le module $M' := M/\bfA m_0$ est de présentation finie.
En effet, 
si $u : E \to F$ est une présentation de $M$ et $y_0 \in F$ un relèvement de $m_0$, alors
$M'$ est présenté par 
$$
u' = u \oplus {\rm mult}_{y_0} : E \oplus \bfA \to F, \qquad
x \oplus \lambda \mapsto u(x) + \lambda y_0
$$
Introduisons l'évaluation en $m_0$ qui va nous permettre de relier le
rang 1 (de $M$) et le rang 0 (de $M'$). Nous la notons~$\eval_{m_0}$ et
nous la voyons indifféremment soit sur $M^\star$ soit sur
$\Ker\transpose u$ (les deux modules sont canoniquement isomorphes) :
$$
\eval_{m_0} : \begin{array}[t]{rcl}
M^\star & \longrightarrow & \bfA \\
\alpha & \longmapsto & \alpha(m_0)
\end{array}
\qquad\qquad \text{ou encore } \qquad\qquad
\eval_{y_0} : \begin{array}[t]{rcl}
\Ker \transpose u & \longrightarrow & \bfA \\
\mu & \longmapsto & \mu(y_0)
\end{array}
$$

\begin{prop}[Du sous-module de Fitting $\FittVect_1(M) \subset M^\star$ à
    l'idéal de Fitting $\calF_0(M/\bfA m_0)$]
\label{Fitting1toFitting0}
\leavevmode

On suppose $\calF_0(M) = 0$. Alors $\eval_{m_0}\big(\FittVect_1(M)\big) = \calF_0(M/\bfA m_0)$.
\end{prop}

\begin{proof}
Reprenons le contexte précédant l'énoncé et notons $r=\dim F-1$ et  $r'=r+1=\dim F$. 
Il s'agit de montrer que $\eval_{y_0}\big(\DVect_r(u)\big) = \calD_{r'}(u')$.

Commençons par l'inclusion $\subset$. Soit $\bfw \in \Im\bigwedge^r(u)$;
l'image de $\bfw^\sharp \in \DVect_r(u)$ par $\eval_{y_0}$ est le déterminant
$[\bfw \wedge y_0]$, qui est bien un mineur d'ordre $r'$ de $u'$.

\smallskip

Quant à l'inclusion réciproque, il y a deux types de mineurs d'ordre $r' = r+1$ de $u'$ :
\begin{itemize}
\item 
ceux construits à partir de $r+1$ colonnes de $u$: ils sont nuls 
car $\calD_{r+1}(u) = 0$ ;
\item
ceux construits à partir de $r$ colonnes de $u$ et de la colonne
$y_0$.  Un tel mineur s'écrit au signe près $[\bfw \wedge y_0]$ avec
$\bfw \in \Im\bigwedge^r(u)$, c'est-à-dire $\bfw^\sharp(y_0)$ où
$\bfw^\sharp \in \DVect_r(u)$.
\end{itemize}%
\end{proof}

\noindent
\textbf{Contexte élimination.} \label{Fitting1toFitting0Elimination}
Soit $\uP$ quelconque. 

Le module $M = \bfB_\delta$ est de présentation finie et $\calF_0(\bfB_\delta) = 0$. On
prend pour $m_0$ la classe d'un  déterminant bezoutien $\nabla$ de sorte
que $\eval_{m_0}$ est l'évaluation en $\overline \nabla$.
La proposition précédente s'écrit alors:
$$
\FittVect_{1}(\bfB_\delta)\text{-évalué-en-$\overline \nabla$}
\ = \ 
\calF_0(\bfB'_\delta)
$$
Remontons à la présentation de $\bfB_\delta$ par $\Syl_\delta$ et
à celle de $\bfB'_\delta$ par $\Syl'_\delta = \Syl_\delta \oplus
\mathrm{mult}_{\nabla}$. L'égalité ci-dessus raconte que l'idéal
constitué des évaluations en $\nabla$ des formes linéaires
déterminantales de $\Syl_\delta$ est égal à l'idéal engendré par les
mineurs pleins de $\Syl'_\delta$:
$$
\DVect_{s_\delta}(\Syl_\delta)\text{-évalué-en-$\nabla$}
\ = \ 
\calD_{s_\delta + 1}(\Syl'_\delta)
$$
Cette égalité est uniquement un \textit{jeu d'écriture} 
une fois acquise la nullité des mineurs d'ordre $s_\delta + 1$ de~$\Syl_\delta$.

\medskip

Dans la suite, va intervenir le sous-module de $M^\star$, dépendant de $m_0$
$$
\Cramer_{m_0}(M) = 
\big\{ \alpha \in M^\star \mid \forall\, m \in M, \ 
\alpha(m_0) m = \alpha(m) m_0
\big\}
$$
Il s'intègre dans le diagramme suivant:
$$
\xymatrix @C=0.5cm {
\FittVect_1(M)\ar[d]_{\eval_{m_0}}  &\subset   &\Cramer_{m_0}(M)\ar[d]|{\eval_{m_0}} &\subset  &M^\star\ar[d]^{\eval_{m_0}}
\\
\calF_0(M/\bfA m_0)         &\subset   &\Ann(M/\bfA m_0)              &\subset  &\bfA
\\
}
$$

\label {NOTA08-Cramerm0M}%
%
%

\begin{prop}[Du sous-module $\Cramer_{m_0}(M)\subset M^\star$ à l'idéal $\Ann(M/\bfA m_0)$]   
\label{Rank1to0}
\leavevmode

Soient $M$ un module de rang 1 et $M' = M/\bfA m_0$ où $m_0 \in M$ est sans torsion.

\begin{enumerate}[\rm i)]
\item 
Le module $M'$ est de rang $0$.

\item 
L'évaluation $\eval_{m_0} : M^\star \to \bfA$ est injective.

\item
Elle induit un isomorphisme $\FittVect_1(M) \simeq \calF_0(M')$.

\item 
Elle induit également un isomorphisme $\Cramer_{m_0}(M) \simeq \Ann(M')$.

En particulier, si $\Cramer_{m_0}(M)$ est monogène engendré par $\vartheta \in M^\star$, alors 
$\Ann(M')$ est monogène engendré par $\vartheta(m_0)$.
\end{enumerate}
\end{prop}

\begin{proof} \leavevmode

On reprend le contexte du début de section où $u : E \to F$ présente $M$,
$y_0 \in F$ relève~$m_0$ et les notations $u'$, $r$ et $r'$.

i)
Pour $a \in \bfA$, montrons l'implication 
$a\calD_{r'}(u')= 0 \Rightarrow a\DVect_{r}(u)= 0$. 
Il en résultera que $\calD_{r'}(u')$ est fidèle puisque $\DVect_r(u)$ l'est.

\smallskip

Soit $\mu \in \DVect_r(u)$. On a $\mu(y_0) \in \calD_{r'}(u')$ donc
$a\mu(y_0) = 0$.  Comme $\mu$ est de Cramer, $a\mu(y) y_0 \in \Im u$
pour tout $y \in F$.  Puisque $m_0$ est sans torsion, $a\mu(y) = 0$.
Ceci ayant lieu pour tout $y \in F$, on en déduit~$a\mu = 0$.

\medskip

ii) Montrons qu'une forme linéaire $\alpha \in M^\star$ telle que $\alpha(m_0) = 0$
est nulle.
L'hypothèse $\alpha(m_0) = 0$ permet de passer $\alpha$ au quotient.
En vertu de~\ref{FormesLineairesRang01} :

\begin{quote}
\it Toute forme linéaire sur un module de présentation finie de rang $0$ est nulle. 
\end{quote}

\noindent
la forme linéaire obtenue sur $M' = M/\bfA m_0$, module de rang $0$, est nulle
et donc $\alpha=0$.

\medskip
iii)
Résulte de $\eval_{m_0}\big(\FittVect_1(M)\big)=\calF_0(M')$ (proposition précédente)
et de l'injectivité de $\eval_{m_0}$.

\medskip
iv)
Inclusion $\eval_{m_0}\big(\Cramer_{m_0}(M)\big) \subset \Ann(M')$.  Soit
$\alpha \in \Cramer_{m_0}(M)$. Montrons l'appartenance $\alpha(m_0) \in
\Ann(M')$.  Par définition de $\Cramer_{m_0}(M)$, la forme linéaire
$\alpha$ vérifie :
$$
\forall\, m \in M, \ \alpha(m_0) m \in \bfA m_0
\qquad \text{d'où} \qquad
\alpha(m_0) \in \Ann(M')
$$
Inclusion réciproque. Soit $a \in \Ann(M')$.
Montrons qu'il existe $\alpha \in \Cramer_{m_0}(M)$ tel que $a = \alpha(m_0)$.
On va construire une forme linéaire $\alpha \in M^\star$ et on montrera 
ensuite qu'elle appartient à $\Cramer_{m_0}(M)$.
Par définition de $a$, on a
$$\forall\, m \in M, \ am \in \bfA m_0$$
En utilisant le fait que $m_0$ est sans torsion, on obtient un 
unique $\lambda_m \in \bfA$ tel que $am = \lambda_m m_0$, ceci pour tout $m \in M$.
On introduit alors l'application $\alpha : M \to \bfA$, $m \mapsto \lambda_m$, et on vérifie 
facilement que cette application est une forme linéaire.
On a alors $am = \alpha(m) m_0$ pour $m\in M$.
En appliquant cette égalité à $m = m_0$ et en utilisant que $m_0$ est sans torsion, on trouve 
$a = \alpha(m_0)$.

En reportant cette information dans l'égalité $am=\alpha(m) m_0$, on obtient 
$$
\forall\, m \in M, \ \alpha(m_0) m = \alpha(m) m_0
$$
ce qui traduit le fait que $\alpha \in \Cramer_{m_0}(M)$.
\end{proof}

\medskip

Dans l'étude du jeu étalon généralisé, nous avons fourni une preuve
\emph {directe} du fait que $\bfB_\delta$ est de MacRae de rang $1$, cf.
\ref{MacRaeJeuEtalonDelta} et, en~\ref{BprimedeltaJeuEtalon}, que
$\bfB'_\delta = \bfB_\delta / \bfA \overline{\nabla}$ est de MacRae de
rang $0$. De plus, nous avons établi le lien précis entre leurs
invariants de MacRae (une forme linéaire d'une part, et un scalaire
d'autre part): l'invariant de MacRae de $\bfB'_\delta$ s'obtient en
évaluant la forme linéaire de $\bfB_\delta$ en $\overline \nabla$
(cf. la remarque~\ref{Rank1to0MacRaeJeuEtalonGen}).

Nous allons voir que ce résultat concernant les invariants de MacRae du couple
$(\bfB_\delta, \bfB'_\delta)$, établi dans le cas très
particulier du jeu étalon généralisé, est valide pour toute suite $\uP$
régulière.  Il utilise d'une part, pour une telle suite, que les
modules étant librement résolubles sont de MacRae (cf. le
théorème~\ref{LibrementResolubleImpliqueMacRae}).  Et d'autre part le
résultat général suivant permettant de relier les invariants de
MacRae d'un module de rang 1 et son quotient par un vecteur sans
torsion.

\begin{prop}[Quotient d'un module de MacRae de rang $1$ par un vecteur sans torsion]
\label{Rank1to0MacRae}
\leavevmode

Soit $M$ un module de MacRae de rang $1$ et $m_0 \in M$ sans torsion.
Alors $M' = M/\bfA m_0$ est un module de MacRae de rang $0$ et
l'évaluation en $m_0$ réalise un isomorphisme:
$$
\MacRaeVect(M) \ \overset{\eval_{m_0}}{\simeq}  \ \MacRae(M')
$$
Ainsi, en notant $\vartheta \in M^\star$ l'invariant de MacRae de $M$, 
celui de $M'$ est $\vartheta(m_0)$.
\end{prop}

\index{invariant de MacRae}%

\begin{proof} 
Notons $\vartheta \in M^\star$ un générateur de $\MacRaeVect(M)$.
Alors, le module $\FittVect_1(M)$ se factorise :
$$
\FittVect_1(M) =  \vartheta\, \fb 
\qquad  \text{avec \quad } \Gr(\fb) \geqslant 2 
$$
En appliquant $\eval_{m_0}$, on obtient $\eval_{m_0}\big(\FittVect_1(M)\big) =
\vartheta(m_0) \fb$.  D'où, avec~\ref{Fitting1toFitting0},
$\calF_0(M') = \vartheta(m_0) \fb$.  L'idéal $\calF_0(M')$ étant
fidèle, cette dernière égalité prouve que le scalaire $\vartheta(m_0)
\in \bfA$ est régulier.  Elle prouve aussi, $\fb$ étant de
profondeur~${\geqslant~2}$, que $\vartheta(m_0)$ est un pgcd fort de
l'idéal $\calF_0(M')$. Ainsi, conformément à la
définition~\ref{DefMacRae0}, $M'$ est de MacRae de rang~$0$,
d'idéal de MacRae $\MacRae(M') = \langle \vartheta(m_0) \rangle$
\end{proof}

\subsection{Les idéaux de Fitting dans la littérature}

Pour convaincre le lecteur de la nécessité de l'intermède initial sur
les idéaux de Fitting, voici, autour de la thématique
\og Les idéaux de Fitting d'un module de présentation finie ne dépendent pas de la
présentation\fg{} quelques références (commentées)
dans la littérature. Il n'est pas rare que le traitement soit relégué en exercice.

$\bullet$
Bourbaki, Algèbre Commutative, chapitre 7 (Diviseurs), exercice 10 de
la section~\S4, p.~106-107. Le mot idéal de Fitting n'est pas prononcé
mais les idéaux déterminantiels sont indexés comme il se doit pour
obtenir l'invariance vis-à-vis de la présentation. L'exercice est
marqué difficile mais c'est peut-être à cause du nombre de questions
(7 questions de (a) à (g)).

$\bullet$
Bruns \& Herzog, Cohen-Macaulay Rings, chap I, section Ideals of minors and Fitting
invariants, p.~21-23: ``the proof is left as an exercice for the reader''.

$\bullet$
Cox, Little, O'Shea, Using Algebraic Geometry, chapter~5, \S4, pages 241-242; tout est
reporté en exercice.

$\bullet$
Eisenbud, Commutative Algebra with a View Toward Algebraic Geometry, section 20.2 (Fitting
Ideals). L'auteur utilise la localisation par tous les idéaux premiers et
mentionne des histoires de présentation minimale.

$\bullet$
Sarah Glaz, Commutative Coherent Rings, chap IV, section 4 Fitting
invariants and Euler Characteristic, p. 96.  Elle dit simplement ``By
Schanuel's lemma 1.1.5, the Fitting ideals of $M$ are independent of
the choice of the presentation''.  Le lemma 1.1.5 en question
(page 2) est le lemme de Schanuel ordinaire (et pas la version améliorée).

$\bullet$
Kunz, Introduction to Commutative Algebra and Algebraic Geometry, le traite en
exercice: exercice~3, page 103 du chapitre IV, section \S1

$\bullet$
Kunz, Kälher Differentials, Appendice D (The Fitting Ideals of a Module), p. 331,
asssure les détails.

$\bullet$
Lang, Algebra, chap XIX (The Alternating Product), section 2 (Fitting ideals),
proposition 2.2 et Lemma 2.3. Les détails y sont assurés.

$\bullet$
Le livre \cite[chap. IV, section 9 (Idéaux de Fitting)]{LombardiQuitte},
fournit, via le lemme 9.2, un plan d'attaque sérieux
mais les détails sont laissés à la lectrice.

$\bullet$
L'ouvrage de Northcott,  \cite[chap. 3 (The invariant of Fitting and MacRae)]{NorthcottFFR},
notamment le théorème~1, page~58, contient un traitement complet.

$\bullet$
Vasconcelos, Computational Methods in Commutative Algebra and
Algebraic Geometry, section 2.4, p.~39. Il énonce le résultat sans preuve
mais en disant ``a fact which follows easily from the Schanuel's lemma
in homological algebra (or by determinantal calculations as originally
done by Fitting)''.

\cleardoublepage

\section{AC$_3$ : Structure multiplicative des résolutions libres finies}
\label{ChapStructureMultiplicative}

Ce chapitre se concentre sur certains complexes de $\bfA$-modules libres de rang fini pour lesquels
nous adoptons les notations suivantes:
$$
\xymatrix @M=0.4pc{
0 \ar[r] &
F_n \ar[r]^-{u_n} &
F_{n-1} \ar[r]^-{u_{n-1}} &
\quad \cdots \quad \ar[r] &
F_1 \ar[r]^-{u_1} &
F_0
}
$$
Nous rappelons la définition du rang attendu $r_k$ de la différentielle $u_k$
(cf. chapitre~\ref{ObjetsSuiteP}, page~\pageref{RangAttendu}) et celle de $r_0$, qui n'est autre,
par définition, que la caractéristique d'Euler-Poincaré $c = c(F_\sbullet)$ du complexe $F_\sbullet$:
$$
\left\{
\begin{array}{l}
r_{n+1} = 0 \\ 
\forall\, k \in \{n, \, {n\!-\!1},\,\dots,1,\, 0\},\quad r_{k+1} + r_k = \dim F_k
\end{array}
\right.
\qquad\qquad
\boxed{c = r_0}
$$
La classe des complexes qui nous intéresse est la classe des complexes
de Cayley (voir la définition page suivante) qui contient en
particulier celle des complexes exacts. Nous avons fait figurer en fin
de chapitre un bref aperçu historique en essayant de préciser les
contributions des auteurs intervenus dans le domaine des
résolutions libres finies.

\label{NOTA09-FreeComplex}%
\label{NOTA09-rk}%
\label{NOTA09-r0}%
\index{rang!attendu (d'une différentielle dans un complexe)}%
\index{caractéristique d'Euler-Poincaré}%

\medskip
Nous allons citer sans preuve deux résultats importants en fournissant
des références précises, non pas chez les initiateurs Buchsbaum \&
Eisenbud mais plutôt dans le livre de Northcott \cite{NorthcottFFR} ou
l'article~\cite{CoquandQuitte} qui sont plus dans l'esprit de notre étude.
En revanche, nous reprenons en charge la notion
de structure multiplicative dans le contexte des complexes de Cayley.
Précisons que les deux dernières références ci-dessus bien qu'elles ne
soient absolument pas de
même nature (le premier est un ouvrage de 270 pages, le second un
article de 14 pages) possèdent un point commun: elles sont
\og self-contained\fg.

\medskip

Dans toute notre étude, nous faisons l'hypothèse $r_k \geqslant 0$ pour
tout $k$.  Ceci est accord avec notre champ d'applications dans lequel
les complexes qui interviennent sont les composantes homogènes du
complexe de Koszul de $\uP = (P_1, \dots, P_n)$. Par ailleurs, dans le cadre plus général
d'un complexe libre \emph {exact}, si l'on~a $r_k < 0$ pour au moins
un $k$, alors l'anneau $\bfA$ est nul.  Ceci n'est pas un résultat
complètement trivial: il est contenu
dans \cite[th. 6.8]{CoquandQuitte}.  Supposer l'anneau de base non
nul n'est pas non plus anodin: c'est le cas des chapitres IV
(Stability and finite free resolutions) ou VI (Grade and finite free
resolutions) de Northcott: ``Throughout the present chapter, $R$ will
denote a non-trivial commutative ring with an identity element''. Ce
qui conduira Northcott à utiliser des idéaux premiers contrairement
à l'article~\cite{CoquandQuitte}.

\medskip
Une preuve du résultat suivant figure en \cite[chap. 4, th.8]{NorthcottFFR} ou en
\cite[th. 6.8]{CoquandQuitte}.

\begin{theo}[Rang stable] \label{StableRank} \leavevmode

Si le complexe $F_\sbullet$ est exact, alors chaque différentielle $u_k$ est de rang $r_k$, au sens où :
$$
\calD_{r_k + 1}(u_k) = 0
\qquad \text{ tandis que } \qquad 
\text{$\calD_{r_k}(u_k)$ est fidèle}
$$
\end{theo}

\index{théorème!du rang stable (complexe libre exact)}%

\medskip
Un autre résultat célèbre est le suivant.

\begin{theo}[What makes a complex exact?] \label{WhatMakesAComplexExact} \leavevmode

\medskip
(i) Si le complexe $F_\sbullet$ est exact, alors $\Gr\big(\calD_{r_k}(u_k)\big) \geqslant k$
pour tout $1 \leqslant k \leqslant n+1$.

\medskip
(ii) Réciproquement, si $\Gr\big(\calD_{r_k}(u_k)\big) \geqslant k$
pour tout $1 \leqslant k \leqslant n+1$, alors le complexe $F_\sbullet$ est exact.
\end{theo}

\index{théorème!what makes a complex exact?}%

On notera que l'inégalité sur la profondeur du point (i) est bien plus précise
que le résultat de fidélité du théorème précédent qui s'énonce
$\Gr\big(\calD_{r_k}(u_k)\big) \geqslant 1$.

En terrain noethérien, une preuve figure dans l'article  \cite[th. I, cor. 1]{BE1}
de D. Buchsbaum \& D. Eisenbud.
Sans aucune hypothèse sur l'anneau (à part non nul chez Northcott), la
preuve du point (i) se trouve en
\cite[th. 6.8]{CoquandQuitte} ou en
\cite[chap. 6, th.14]{NorthcottFFR}.
Quant à la réciproque (ii), cf \cite[cor. 7.5]{CoquandQuitte} ou~\makebox{\cite[chap. 6, th.15]{NorthcottFFR}}.

\subsection{Structure multiplicative d'un complexe de Cayley}
\label{subsectionStructureMult}

Soit $F_\sbullet$ un complexe libre de caractéristique
d'Euler-Poincaré $c$.  Moyennant une hypothèse adéquate sur
$F_\sbullet$, l'objectif est d'associer à ce complexe une forme
$c$-linéaire alternée sans torsion $\mu$ sur $F_0$, définie à un inversible
près.  Le $\bfA$-module engendré par cette forme $\mu$ est donc un
sous-module $\CayleyVect(F_\sbullet) \subset \BW^c(F_0)^\star$, libre
de rang~1: c'est le \emph {déterminant de Cayley} de $F_\sbullet$.

\medskip 
L'hypothèse adéquate consiste à affaiblir la notion de complexe exact de la manière suivante.

\begin{defn} \label{DefComplexeCayley}
On considère un complexe de $\bfA$-modules libres de rang fini 
$$
\xymatrix @M=0.4pc{
0 \ar[r] &
F_n \ar[r]^-{u_n} &
F_{n-1} \ar[r]^-{u_{n-1}} &
\quad \cdots \quad \ar[r] &
F_1 \ar[r]^-{u_1} &
F_0
}
$$
On dit que le complexe est de Cayley lorsque 
$\Gr\big(\calD_{r_k}(u_k)\big) \geqslant 2$ 
pour $k \geqslant 2$ et \mbox{$\Gr\big(\calD_{r_1}(u_1)\big) \geqslant 1$}.
\end{defn}

\medskip
Il s'agit bien d'un affaiblissement: un complexe exact de modules
libres vérifie ${\Gr\big(\calD_{r_k}(u_k)\big) \geqslant k}$, cf. le
point (ii) du théorème \ref{WhatMakesAComplexExact}
\og What makes a complex exact \fg{} et de ce fait
est un complexe de Cayley.

\medskip
Nous verrons que la construction de la forme $c$-linéaire $\mu$ 
(dont il est question avant la définition) est
étroitement liée à une factorisation en rang $r_1$ de la première
différentielle $u_1$, factorisation décrite par 
le diagramme commutatif
$$
\vcenter {\hbox{
\xymatrix @M= 0.4pc{ 
\bigwedge^{r_1}(F_1) \ar[rd]_-{\nu_1} \ar[rr]^-{\bigwedge^{r_1}(u_1)} & & 
\bigwedge^{r_1}(F_0) \\
& \bfA \ar[ru]_-{\times \Theta_1} & 
}}}
\qquad
\begin {array}{l}
\hbox{et nous aurons $\mu = \Theta_1^\sharp$ où $\sharp = \sharp_0$ est un isomorphisme} \\ [2mm]
\hbox {déterminantal $\sharp_0 : \BW^{r_1}(F_0) \simeq \BW^{c}(F_0)^\star$} \\
\end{array}
$$
Cette factorisation s'accompagne d'une précision supplémentaire sur les profondeurs
de $\nu_1$ et $\Theta_1$ qui va faire que le module $M := \Coker(u_1)$ présenté
par $u_1$ est un module de MacRae de rang $c$. On rejoint et complète ainsi
les notions/informations du chapitre précédent.

\index{isomorphisme déterminantal}%
\index{factorisation $\text{colonne}\times\text{ligne}$ (puissance extérieure d'une application linéaire)}%

\medskip

Terminons cette présentation en apportant quelques
précisions. Pour toute différentielle $u_k$, et pas seulement $u_1$, on va disposer d'une
factorisation de la puissance extérieure $\BW^{r_k}(u_k)$ sous la
forme \og vecteur $\times$ forme linéaire \fg{}.  La mise en place de
ces factorisations va nécessiter des choix d'isomorphismes déterminantaux
sur les $(F_k)_{k \geqslant 0}$. Mais le déterminant de Cayley $\CayleyVect(F_\sbullet)$
de $F_\sbullet$ (module libre de rang $1$)
est indépendant de ces choix.

\emph {Tout} complexe libre~$F_\sbullet$ possède un déterminant
(attention, ce n'est pas le déterminant de Cayley) sorte de produit
tensoriel alterné des puissances extérieures de rang maximum des termes 
$$
\textstyle
\Det(F_\sbullet) \ = \ 
\bigotimes\limits_{k \, \rm pair} \bigwedge^{\dim  F_k}(F_k)
\ \otimes \ \!\!
\bigotimes\limits_{k \, \rm impair} \!\!
\bigwedge^{\dim  F_k}(F_k)^\star
\, \ = \ \, 
\bigwedge^{\dim  F_0}(F_0) \, \otimes \, \bigwedge^{\dim  F_1}(F_1)^\star \, \otimes \, 
\bigwedge^{\dim  F_2}(F_2) \, \otimes \, \cdots
$$
Ce $\bfA$-module $\Det(F_\sbullet)$, qui est évidemment libre de rang~$1$, ne dépend que
des termes $F_k$ du complexe (et pas des différentielles).
L'hypothèse supplémentaire \og $F_\sbullet$ de Cayley \fg{} 
permettra la
mise en place d'un morphisme injectif intrinsèque $\rho
:\Det(F_\sbullet)^\star \hookrightarrow \BW^c(F_0)^\star$, établissant un
isomorphisme $\Det(F_\sbullet)^\star \simeq \CayleyVect(F_\sbullet)$.

\bigskip

Le théorème suivant est énoncé dans l'article~\cite{CoquandQuitte}
avec des termes $F_k$ équipés de base, et seule la partie existence de
la factorisation est démontrée.  Dans la formulation qui vient, nous
n'utilisons pas de bases seulement des isomorphismes déterminantaux
fixés sur les termes.

A cette occasion, on rappelle qu'un isomorphisme déterminantal sur un
module libre $F$ dépend d'une orientation $\bff$ de $F$. Dans la
suite, il sera sous-entendu que cette orientation $\bff$ est \emph
{partie intégrante} de l'isomorphisme déterminantal.  Ce dernier dépend
également de la \og position de la colonne-argument\fg{} et nous pouvons
par exemple choisir de mettre le joker à droite
$$
\sharp : 
\begin{array}[t]{rcl}
\bigwedge^{r}(F) & \longrightarrow & \bigwedge^{r'}(F)^\star \\ [0.4em]
\Theta & \longmapsto & \Theta^\sharp := 
\oriented{\Theta\wedge \sbullet\,}_\bff
\end{array}
\qquad 
\text{avec $r + r' = \dim F$}
$$
mais rien ne nous y oblige. Cependant \emph {pour les calculs}, nous prendrons
toujours cette même position de joker à droite.

\medskip 
Ces précisions étant apportées, voici l'énoncé que l'on obtient :

\index{structure multiplicative d'un complexe de Cayley}%

\begin{theo}[Structure multiplicative, à isomorphismes déterminantaux imposés] \label{Factorisation}

\leavevmode

Considérons un complexe de Cayley :
$$
\xymatrix @M=0.4pc{
0 \ar[r] &
F_n \ar[r]^-{u_n} &
F_{n-1} \ar[r]^-{u_{n-1}} &
\quad \cdots \quad \ar[r] &
F_1 \ar[r]^-{u_1} &
F_0
}
$$
Pour $k \geqslant 1$, on munit $F_k$ d'un isomorphisme déterminantal $\sharp_k
: \BW^{r_{k+1}}(F_k) \buildrel \simeq \over \longrightarrow \BW^{r_k}(F_k)^\star$.
Relativement à ces isomorphismes,  il existe une suite bien déterminée de vecteurs
$(\Theta_k)_{1 \leqslant k \leqslant n+1}$ avec $\Theta_k \in \BW^{r_k}(F_{k-1})$
et $\Theta_{n+1} = 1 \in \bigwedge^{0}(F_n)$,  telle que

\begin{enumerate}[\rm i)]
\item 
$\bigwedge^{r_k}(u_k)$ soit factorisé via ce diagramme commutatif :
$$
\vcenter{
\xymatrix @M= 0.4pc{ 
\bigwedge^{r_k}(F_k) \ar[rd]_-{\Theta_{k+1}^{\sharp}} \ar[rr]^-{\bigwedge^{r_k}(u_k)} & & 
\bigwedge^{r_k}(F_{k-1}) \\
& \bfA \ar[ru]_-{\times \Theta_k} & 
}}
\quad 
\qquad 
\textstyle \bigwedge^{r_k}(u_k) 
\ = \ \Theta_k\,\Theta_{k+1}^\sharp, \text{où $\sharp = \sharp_k$}
$$

\item
on ait les inégalités de profondeur 
$\Gr(\Theta_k) \geqslant 2$ pour $k \geqslant 2$ et $\Gr(\Theta_1) \geqslant 1$.
\end {enumerate}

On dira que $(\Theta_k)_{1 \leqslant k \leqslant n+1}$ est le système de factorisation de $F_\sbullet$
déduit de $(\sharp_k)_{1 \leqslant k \leqslant n}$.
\end{theo}

\label{NOTA09-bfe}%
\label{NOTA09-sharp}%
\label{NOTA09-Theta}%
%
%

\begin{proof}

En ce qui concerne l'unicité, supposons $\Theta_{k+1}$ déterminé
avec $\Gr(\Theta_{k+1}) \geqslant 2$, a fortiori $\Theta_{k+1}$ sans
torsion. La forme linéaire $\Theta_{k+1}^\sharp$
est également sans torsion donc pour deux vecteurs $\Theta_k, \Theta'_k$
du module libre $\BW^{r_k}(F_{k-1})$, on a l'implication:
$$
\Theta_k\,\Theta_{k+1}^\sharp = \Theta'_k\,\Theta_{k+1}^\sharp \quad\Rightarrow\quad
\Theta_k = \Theta'_k
$$
ce qui assure l'unicité.

\medskip
\noindent
En ce qui concerne l'existence, nous avons choisi de la reporter en
annexe (\ref {sectionAnnexeStructureMultiplicative}).  Ici, on peut
juste en évoquer le schéma: puisque $\Theta_{n+1} = 1$, on factorise
sans peine la puissance extérieure maximale de la dernière
différentielle $u_n$ (à gauche sur les schémas) et, par récurrence, on
descend jusqu'au degré homologique~1. La propagation de la
factorisation repose essentiellement sur la notion de profondeur
$\geqslant 2$ et sur un théorème de proportionnalité découlant
d'identités de Plücker-Sylvester entre mineurs.
\end{proof}

\medskip

Cette propriété de factorisation de la puissance extérieure (de rang
convenable) de chaque différentielle est une propriété extrêmement
profonde, notamment celle de la \textit{première} différentielle.  En rappelant
que~$c$ est la caractéristique d'Euler-Poincaré de $F_\sbullet$, la
factorisation $\bigwedge^{r_1}(u_1) = \Theta_1\,\Theta_2^\sharp$ donne naissance
à une forme $c$-linéaire alternée sur le premier terme $F_0$.
C'est le vecteur $\Theta_1$, de profondeur $\geqslant 1$, 
qui donne naissance à cette forme $c$-linéaire.

\begin {defn}[Forme $c$-linéaire, à isomorphismes déterminantaux imposés] 
\label{cLinearFormOfComplex}
On garde le contexte précédent et on munit $F_0$ d'un isomorphisme déterminantal~$\sharp_0$.
On définit une forme $c$-linéaire alternée 
 $\mu \in \BW^c(F_0)^\star$, associée
au système $(\sharp_k)_{0 \leqslant k \leqslant n}$, par 
$\mu = \Theta_1^{\sharp_0}$.
Cette forme $c$-linéaire $\mu$ est sans torsion.
\end {defn}

\label{NOTA09-mu}%
%
%

\begin{rmq}

\`A partir d'un système de factorisation de $F_\sbullet$ par des 
\textit{vecteurs} $\Theta_k \in \bigwedge^{r_k}(F_{k-1})$ 
pour $k \in \llbracket 1, n+1\rrbracket$, on peut construire un
système de factorisation de $F_\sbullet$ par des \textit{formes
linéaires} $\nu_k \in \bigwedge^{r_k}(F_{k})^\star$ pour
$k \in \llbracket 0, n \rrbracket$, via $\nu_k = \Theta_{k+1}^\sharp$.
Et il y a évidemment une correspondance entre ces deux systèmes, 
dans le sens où $\Theta_{k} = \nu_{k-1}^{\flat}$ en notant~$\flat$ l'isomorphisme
déterminantal inverse $\bigwedge^{r_{k-1}}(F_{k-1})^\star \to \bigwedge^{r_{k}}(F_{k-1})$. 
Ainsi, 
$\bigwedge^{r_k}(u_k) = \Theta_k \, \Theta_{k+1}^\sharp$
peut être également lu en 
$\bigwedge^{r_k}(u_k) = \nu_{k-1}^{\flat} \, \nu_k$.

La suite de \textit{vecteurs} 
$\Theta_k$ s'est construite à partir du degré homologique $n+1$ 
en posant $\Theta_{n+1} = 1$.
{\it A contrario}, la suite de \textit{formes linéaires} $\nu_k$ est obtenue à partir du degré homologique $1$, 
à partir de la fameuse forme $c$-linéaire alternée $\mu$ intervenant dans la définition précédente.
On a alors une forme linéaire $\nu_1 \in \bigwedge^{r_1}(F_1)^\star$ 
telle que $\bigwedge^{r_1}(u_1) = \mu^\flat\,\nu_1$ ; et on peut ainsi reconstruire toute 
la factorisation du complexe.
\end{rmq}

\begin {rmq}
Nous utiliserons essentiellement cette notion de structure multiplicative de complexe de
Cayley dans le cadre d'un complexe exact à savoir la composante homogène de degré $d$
(variable) du complexe de Koszul d'une suite régulière~$\uP$.
Nous avons cependant tenu à présenter cette notion de complexe de Cayley
car elle nous semble être un bon cadre permettant de mettre en place le
déterminant de Cayley.

\medskip
\noindent
Bien entendu, il existe des complexes de Cayley qui ne sont pas
exacts. Prenons le cas du complexe de Koszul $\rmK_\sbullet(\ua)$
d'une $n$-suite de scalaires $\ua = (a_1, \dots, a_n)$. En posant
$r_k = \binom{n-1}{k-1}$, on a $r_{k+1} + r_k = \binom{n}{k}
= \dim \BW^k(\bbA^n)$ et $r_{n+1} = 0$ prouvant ainsi que $r_k$ est le
rang attendu de la différentielle $\partial_k$.

\noindent
Fixons $n=3$. Voici le complexe $\rmK_\sbullet(\ua)$
$$
0 \to \xymatrix {\bigwedge^3(\bbA^3)\ar[r]^{\partial_3} & \bigwedge^2(\bbA^3)
\ar[r]^{\partial_2} &\bigwedge^1(\bbA^3)
\ar[r]^{\partial_1} &\bigwedge^0(\bbA^3)}
$$
avec comme rangs attendus:
$$
0 \to \xymatrix {\bbA\ar[r]^{\partial_3}_1 &\bbA^3\ar[r]^{\partial_2}_2 &\bbA^3 
\ar[r]^{\partial_1}_1 &\bbA}
$$
En utilisant la base $(e_2\wedge e_3, e_3\wedge e_1, e_1\wedge e_2)$ pour
$\bigwedge^2(\bbA^3)$, on a~:
$$
\partial_3 = \begin {bmatrix} a_1\\ a_2\\ a_3 \end {bmatrix}
\qquad
\partial_2 = \begin {bmatrix} 0&a_3 & -a_2\\ -a_3 &0 &a_1\\ a_2 &-a_1 &0 \end{bmatrix}
\qquad
\partial_1 = \begin {bmatrix} a_1 & a_2 & a_3 \end {bmatrix}
$$
Ce qui donne comme idéaux déterminantiels:
$$
\calD_1(\partial_3) = \langle\ua\rangle, \qquad
\calD_2(\partial_2) = \langle\ua\rangle^2, \qquad
\calD_1(\partial_1) = \langle\ua\rangle.
$$
Pour obtenir un complexe de Cayley non exact, il suffit de prendre $\ua$ vérifiant $\Gr(\ua) = 2$.
\end {rmq}

\bigskip
Dans un complexe de Cayley $F_\sbullet$, il n'est pas évident a priori
que $\calD_{r_k+1}(u_k) = 0$ mais c'est bien le cas.  Cela se déduit
de la nature de la factorisation de $\BW^{r_k}(u_k)$.  Nous en
reportons la preuve plus loin, cf. la proposition~\ref{StableRankFromFactorization}
mais nous énonçons ici le résultat sous la forme:

\begin{theo}[Rang des différentielles]  

Dans un complexe de Cayley, chaque différentielle $u_k$ est de rang $r_k$, au sens où :
$$
\calD_{r_k + 1}(u_k) = 0
\qquad \text{ et } \qquad 
\text{$\calD_{r_k}(u_k)$ est fidèle}
$$
\end{theo}

\subsection{Peut-on illustrer la factorisation via un exemple?}

Dans le contexte de l'élimination, c'est bien entendu la composante
homogène $\rmK_{\sbullet,d}$ de degré $d$ du complexe de Koszul de la
suite $\uP$ qui est au coeur du sujet.  Si la suite $\uP$ est
régulière, ce complexe de $\bfA$-modules libres est exact, et, en
chaque degré homologique~$k$, la puissance extérieure d'ordre
$r_{k,d}$ de la différentielle $\partial_{k,d}(\uP)
: \rmK_{k,d} \rightarrow \rmK_{k-1,d}$ se factorise sous la forme \og
vecteur $\times$ forme linéaire \fg{}.

Le cas $(k=1, d=\delta)$, pour lequel $r_{1,\delta}
= \dim \bfA[\uX]_\delta-1$, va jouer un rôle tout à fait remarquable.
En rappelant la notation $s_\delta = r_{1,\delta}$, on dispose d'une
factorisation de $\bigwedge^{s_\delta}(\Syl_\delta) =
\Theta\,\nu$ où $\Theta \in \BW^{s_\delta}(\bfA[\uX]_\delta)$ et
$\nu : \BW^{s_\delta}(\rmK_{1,\delta}) \to \bfA$. \`A partir de $\Theta$
et d'une orientation $\bff$ de $\rmK_{0,\delta} \overset{\rm def.}{=} \bfA[\uX]_\delta$, 
on peut élaborer la forme linéaire $\Theta^\sharp = [\Theta \wedge\sbullet]_\bff :
\bfA[\uX]_\delta \to \bfA$. Une telle factorisation, unique à un inversible près, dépend 
de choix de systèmes d'orientations des termes de
$\rmK_{\sbullet,\delta}$. On verra dans le
chapitre \ref{ComplexeDecompose} comment sélectionner certains systèmes
d'orientations de manière à obtenir une \emph {unique} forme linéaire $\omegares = \Theta^\sharp$.

\medskip

Il est difficile, voire impossible, d'illustrer \textit{directement}
cette factorisation vu sa complexité.
Par exemple, dans le cas très modeste du format $D=(2,2,2)$, on a 
$\delta = 3$, $s_\delta +1 = 10$, $\omegares=\omega$, et, dans le cas générique 
(\idest{} $\bfA$ anneau de polynômes 
à $3 \times 6 = 18$ indéterminées qui sont les coefficients des~$P_i$), voici 
le nombre de monômes (en les coefficients des $P_i$) 
de chaque $\omegares(X^\alpha)$ pour les $10$ monômes~$X^\alpha$
de degré~$\delta$ :
$$
2238,\  1650,\  1650,\  1650,\  1302,\  1650,\  2238,\  1650,\  1650,\  2238
$$
Quant au scalaire $\omegares(\nabla)$, le résultant de $\uP$ quand tout sera
en place, il comporte 21894 monômes !

\bigskip

Nous proposons donc un exemple plus modeste de factorisation 
d'application linéaire issue d'une résolution libre.
Il s'agit de la deuxième différentielle $u = u_2 : \bfA^4 \to \bfA^3$ de la résolution libre minimale 
de l'idéal $\langle x^2, yt, z(x+y) \rangle$ de $\bfA = \bbZ[x,y,z,t]$ :
$$
\xymatrix @M=0.4pc{
0 \ar[r] & \bfA^2 \ar[r]^-{u_3} & \bfA^4 \ar[r]^-{u} & \bfA^3 \ar[r]^-{u_1} & \bfA
}
$$
Posons $s = x+y$ et notons $(e_j)$ une base de $E = \bfA^4$ et $(f_i)$ une base de $F= \bfA^3$.
Alors:
$$
u_3 =
\left[
\begin{array}{*{2}{c}}
. & x \\ 
x & -y \\ 
-t & . \\ 
z & z \\ 
\end{array}
\right]
\qquad\qquad
u = \NorthEastBordermatrix{
\Heti{\,e_{1}} & \Heti{\,e_{2}} & \Heti{\,e_{3}} & \Heti{\,e_{4}} \\
. & zt & zs & yt & \Heti{\,f_{1}} \\ 
zs & xz & . & -x^{2} & \Heti{\,f_{2}}\\ 
-yt & -xt & -x^{2} & . & \Heti{\,f_{3}} \\ 
}
\qquad\qquad
u_1 = \left[
\begin{array}{*{3}{c}}
x^{2} & yt & zs \\ 
\end{array}
\right]
$$
L'application linéaire $u$ est de rang attendu $2$. 
Quant à $\bigwedge^2(u)$, c'est l'application linéaire 
représentée par la matrice 
$$
\NorthEastBordermatrix{
\Heti{\  e_{12}} & \Heti{\quad  e_{13}} & 
\Heti{\quad e_{14}} & \Heti{\quad e_{23}} & \Heti{\quad e_{24}} & \Heti{\quad e_{34}} & \\
-z^{2}st & -z^{2}s^{2} & -yzst & -xz^{2}s & -xzst & -x^{2}zs & \Heti{\,f_{12}} \\ 
yzt^{2} & yzst & y^{2}t^{2} & xyzt & xyt^{2} & x^{2}yt & \Heti{\,f_{13}} \\ 
-x^{2}zt & -x^{2}zs & -x^{2}yt & -x^{3}z & -x^{3}t & -x^{4} & \Heti{\,f_{23}} \\ 
}
\text{dont chaque colonne est multiple de\ }
\begin{bmatrix}
zs  \\
-yt \\
x^2 \\
\end{bmatrix}.
$$
Précisément, on a la factorisation $\bigwedge^2(u) = \Theta \, \nu$ 
avec $\Theta = zs\,f_{12} - yt \, f_{13} + x^2 \,f_{23} \in \bigwedge^2(F)$
et $\nu = -(zt e_{12}^\star + zs e_{13}^\star + yt e_{14}^\star + xz e_{23}^\star
+ xt e_{24}^\star + x^2 e_{34}^\star) \in \bigwedge^2(E)^\star$.

Cette factorisation (ou plutôt la forme linéaire $\nu$)
provient de la factorisation $\bigwedge^{r_3}(u_3) = \Theta_3\,\nu_3$.
Ici $r_3 = 2$ et $\nu_3 : \bigwedge^{2}(\bfA^2) \rightarrow \bfA$ 
est la forme linéaire qui envoie une orientation de $\bfA^2$ sur $1$.
Quant au vecteur $\Theta_3 \in \bigwedge^2(E)$, c'est 
$\Theta_3 = -x^2 e_{12} + xt e_{13} - xz e_{14} - yt e_{23} + zs e_{24} - zt e_{34}$.
Et la forme linéaire~$\nu$ intervenant dans la factorisation de $\bigwedge^2(u)$ 
est l'image de $\Theta_3$ par un isomorphisme déterminantal~$\sharp$. 
Si on choisit pour~$\sharp$ l'isomorphisme $\bfw \mapsto \oriented{\bfw\wedge\sbullet}_\bfe$ avec
$\bfe = e_1 \wedge e_2 \wedge e_3 \wedge e_4$, alors: 
$$
\Theta_3^\sharp 
\ =\   
-x^2 e_{12}^\sharp + xt e_{13}^\sharp  - xz e_{14}^\sharp  - yt e_{23}^\sharp + 
zs e_{24}^\sharp  - zt e_{34}^\sharp
\ =\ 
-x^2 e_{34}^\star - xt e_{24}^\star  - xz e_{23}^\star  - yt e_{14}^\star - 
zs e_{13}^\star  - zt e_{12}^\star
$$
ce qui est bien la forme linéaire $\nu$ précédente.

\subsection {Extension des scalaires d'un système de factorisation}

\begin{prop}[Extension des scalaires] \label{ExtensionScalaires}
Soit $\kappa : \bfA \to \bfA'$ un morphisme d'anneaux, 
$F_\sbullet$ un complexe de Cayley sur $\bfA$, 
et $\kappastar(F_\sbullet)$ le complexe obtenu à partir de $F_\sbullet$ par extension des scalaires.
Si $\kappastar(F_\sbullet)$ est un complexe de Cayley,
alors un système de factorisation de $\kappastar(F_\sbullet)$ est obtenu en prenant 
l'image par $\kappastar$ d'un système de factorisation de $F_\sbullet$.
\end{prop}

\begin{proof}
Notons avec des primes les images des objets par $\kappastar$: $F'_k
= \bfA' \otimes_\bfA F_k$ et $u'_k = \id_{\bfA'} \otimes_\bfA u_k$.
Comme $\dim\nolimits_{\bfA}(F_k) = \dim\nolimits_{\bfA'}(F'_k)$, le
rang attendu de $u'_k$ n'est autre que $r_k$, celui de $u_k$.  Enfin,
pour des raisons typographiques, nous continuons à noter $\sharp$ au
lieu de $\sharp'$ les isomorphismes déterminantaux
de~$\kappastar(F_\sbullet)$ déduits de ceux de $F_\sbullet$.

\medskip
\noindent
En appliquant $\kappastar$ à la factorisation
$\bigwedge\nolimits^{r_k}(u_k) = \Theta_k\ \Theta_{k+1}^{\sharp}$ de
$F_\sbullet$, on obtient $\bigwedge\nolimits^{r_k}(u'_k)
= \Theta'_k\ \Theta_{k+1}'^{\sharp}$.  Il reste à vérifier que ces
vecteurs $\Theta'_k$ ont la profondeur désirée.

\medskip
\noindent
$\rhd$
Première preuve, en utilisant uniquement la définition \og Cayley \fg{}.
On a l'inclusion :
$$
\calD_{r_k}(u'_k) 
\ \subset \ 
\langle \text{composantes de } \Theta'_k \rangle 
$$
Par hypothèse, $\kappastar(F_\sbullet)$ est de Cayley, 
donc $\Gr\big(\calD_{r_k}(u'_k)\big) \geqslant 2$ pour tout $k \geqslant 2$ 
et $\Gr\big(\calD_{r_1}(u'_1)\big) \geqslant 1$. 
Puisque $\fb\supseteq\fa \Rightarrow \Gr(\fb)\geqslant \Gr(\fa)$,
cf.~\ref{Gr2ProprietesElementaires}, on a $\Gr(\Theta'_k) \geqslant
2$ pour tout $k \geqslant 2$ et $\Gr(\Theta'_1) \geqslant 1$.

\medskip
\noindent
$\rhd$
Deuxième preuve, en utilisant la factorisation \og Cayley\fg{}.
Par hypothèse, $\kappastar(F_\sbullet)$ est de Cayley, donc possède
un système de factorisation $(\theta'_k)_k$ avec
les profondeurs désirées et $\theta'_{n+1} = 1$ où
$n$ est la longueur de $F_\sbullet$.

Raisonnons par récurrence descendante pour montrer que $\Theta'_k
= \theta'_k$ ce qui impliquera que $\Theta'_k$ possède la profondeur
adéquate.  Par définition des systèmes de factorisation, on a
$\Theta'_{n+1} = \theta'_{n+1} = 1$ d'où l'initialisation de la récurrence.
Considérons:
$$
\textstyle
\bigwedge\nolimits^{r_k}(u'_k) 
\ =\ 
\Theta'_k\ \Theta_{k+1}'^{\sharp}
\ = \ 
\theta'_k \ \theta_{k+1}'^{\sharp}
$$
et supposons avoir prouvé $\Theta'_{k+1} = \theta'_{k+1}$. Puisque $\theta'_{k+1}$ est
sans torsion, on en déduit $\Theta'_k = \theta'_k$.
\end {proof}

\begin{rmq}
Nous utiliserons essentiellement ce résultat dans le cadre de l'élimination de
la manière suivante.
On dispose d'un système homogène $\uP$ à coefficients dans $\bfA$. Lorsque la suite $\uP$
est régulière, $\rmK_{\sbullet,d}(\uP)$ est un complexe libre exact.
Imaginons que l'on souhaite démontrer certains résultats relatifs à $\uP$ et au
système de factorisation de $\rmK_{\sbullet,d}(\uP)$. On peut alors se
permettre de génériser~$\uP$ en prenant pour indéterminées certains
coefficients des $P_i$, voire tous. On dispose ainsi d'un autre 
anneau~$\bfA^\gen$ et d'un morphisme $\kappa : \bfA^\gen \to \bfA$.
On peut alors se limiter à montrer le résultat pour $\uP^\gen$ générique,
ce qui le prouve pour $\uP$ via la proposition précédente.
\end{rmq}

\subsection{Factorisation $\BW^r(u)=\hbox{vecteur}\times\hbox{forme linéaire}$ \og pour elle-même\fg}

\index{factorisation $\text{colonne}\times\text{ligne}$ (puissance extérieure d'une application linéaire)}%

\noindent
\begin{minipage}[c]{0.6\linewidth}
Le \og pour elle-même\fg{} du titre a la signification suivante: on met provisoirement de côté
les complexes de Cayley et on se donne
\emph{une} application linéaire $u : E \to F$ entre deux modules libres
ainsi qu'\emph{une} factorisation $\bigwedge^r(u) = \Theta\,\nu$ 
où $\Theta \in \bigwedge^r(F)$ et
$\nu : \bigwedge^r(E) \rightarrow \bfA$.
Modulo des hypothèses supplémentaires de régularité sur $\Theta,\nu$,
quelles en sont les conséquences déterminantales?
L'objectif est l'égalité $\calD_{r+1}(u) = 0$ 
(cf.~\ref{StableRankFromFactorization})
et la coïncidence entre la factorisation $\bigwedge^r(u) = \Theta\,\nu$ 
et celle de $\DVect_r(u)$ sous la forme $\DVect_r(u)=\Theta^\sharp\,\fb$ où $\fb$ est l'idéal $\Im\nu$
(cf.~\ref{FactorisationEtPgcdFort}).
\end{minipage}
\hfill
\begin{minipage}[c]{0.4\linewidth}
$
\begin {array}{c}
\xymatrix @C=0.7cm @R=1.1cm @M=0.4pc{
\BW^{r}(E) \ar[d]_-{\nu} \ar[r]^-{\BW^{r}(u)} 
                 &\BW^{r}(F)\ar[r]_{\sharp}^\simeq &\BW^c(F)^\star \\
\bfA \ar[ru]_-{\times\Theta}
} \\
c = \dim F - r \\
\end {array}
$
\end{minipage}


\bigskip

L'attention portée à ce type de factorisation est bien entendue guidée
par le fait qu'elle intervient dans tout complexe de Cayley 
(a fortiori dans tout complexe libre exact) pour
chaque différentielle, la puissance extérieure étant celle de rang
attendu. De manière
indépendante, voici deux exemples élémentaires de factorisation,
presque triviaux.

\smallskip
\noindent
$\rhd$ 
Le premier est celui d'une application injective 
avec $r = \dim E$.
Alors $\bigwedge^r(u) : \bigwedge^r(E) \simeq \bfA \rightarrow \bigwedge^r(F)$ 
s'identifie, après le choix d'une base $\bfe \in \bigwedge^r(E)$, à la multiplication par 
$\Theta = \bigwedge^r(u)(\bfe)$, vecteur sans torsion car $u$ est injective.
Autrement dit, on a la factorisation $\bigwedge^r(u) = \Theta\, \bfe^\star$.

\noindent
$\rhd$ 
Le second est celui d'une application surjective avec $r = \dim F$.
Alors $\bigwedge^r(u) : \bigwedge^r(E) \rightarrow \bigwedge^r(F) \simeq \bfA$ est surjective 
et s'identifie, après le choix d'une base $\bff \in \bigwedge^r(F)$, 
à la forme linéaire (surjective donc) 
$\nu = \bff^\star \circ \bigwedge^r(u) \in \bigwedge^r(E)^\star$.
Autrement dit, on a la factorisation $\bigwedge^r(u) = \bff \, \nu$.

\bigskip
On rappelle, pour un vecteur $v$, que l'inégalité $\Gr(v) \geqslant 1$ est synonyme de
$v$ sans torsion ou encore du fait que l'idéal contenu $\rmc(v)$ est un idéal fidèle.

\begin{prop}[Factorisation et rang] \label{StableRankFromFactorization} \leavevmode

\noindent
Soit $u : E \to F$ et une factorisation $\bigwedge^r(u) = \Theta\, \nu$ où
$\Theta \in \bigwedge^r(F)$ et $\nu : \bigwedge^r(E) \rightarrow \bfA$.

\begin{enumerate}[\rm i)]
\item 
Si $\Gr(\nu,\Theta) \geqslant 1$ alors $\bigwedge^{r+1}(u) = 0$, 
c'est-à-dire $\calD_{r+1}(u) = 0$.
\item
On a la factorisation $\calD_r(u) = \rmc(\Theta)\rmc(\nu)$.
\item
Si $\Gr(\nu) \geqslant 1$ et $\Gr(\Theta) \geqslant 1$, l'idéal $\calD_r(u)$ est fidèle.
\end{enumerate}

Par conséquent, si $\Gr(\nu) \geqslant 1$ et $\Gr(\Theta) \geqslant 1$, 
alors $u$ est de rang $r$.
\end{prop}

\begin {proof} \leavevmode

i) Soit $U$ la matrice de $u$ dans des bases $(e_j)$ de $E$ et $(f_i)$ de $F$.
Ecrivons $\Theta = \sum_{\#I = r} \Theta_I f_I$ et posons $\nu_J
= \nu(e_J)$ de sorte que $\det_{I\times J}(U) = \Theta_I\, \nu_J$.

\smallskip
\noindent
Pour toute partie $K$ de cardinal $r$, on va montrer que l'on a 
$\nu_K^2\,\bigwedge^{r+1}(u)=0$ ; 
par transposée, il en sera de même pour $\Theta_K^2$ à la place de $\nu_K^2$ 
et l'hypothèse $\Gr(\Theta,\nu) \geqslant 1$ permettra de conclure.

\noindent
On doit démontrer que $\nu_K^2\, \det_{I'\times J'}(U) = 0$ pour $I'$ et $J'$ de cardinal $r+1$.
La partie $I'$ étant fixée, on note ci-dessous 
$[\quad]$ le déterminant de $r+1$ vecteurs de $F$ extraits sur les lignes d'indice $I'$. 

\smallskip
\noindent
Remarquons d'abord que pour $X \in F$ et $J$ de cardinal~$r$, on a : 
$$
\nu_K [X, U_J] = \nu_J [X, U_K]
$$
Il suffit en effet de développer chaque déterminant de part et d'autre
de l'égalité selon la première colonne et d'utiliser $\det_{I\times
J}(U) = \Theta_I\,\nu_J$ et $\det_{I\times K}(U) = \Theta_I\,\nu_K$.

\noindent
En conséquence, pour tout $\ell$ :
$$
\nu_K [U_\ell,U_J] = \nu_J [U_\ell,U_K]
\leqno (\star_{\ell,J})
$$ 
Soit $j$ le premier élément de $J'$, $J$ son complémentaire de sorte que $J' = j \vee J$ et $\#J = r$.
De la même manière, soit $k$ le premier élément de $K$, $L$ son complémentaire  de sorte que  
$K = k \vee L$ et $\#L = r-1$. On a alors
$$
\begin{array}{rcl}
\nu_K [U_j,U_K] 
&\buildrel{\rm def}\over=& \nu_K [U_j,U_k,U_L] =  -\nu_K[U_k,U_j,U_L]
\\
&\buildrel{\rm def}\over=&
-\nu_K [U_k,U_{j\vee L}] \buildrel {(\star_{k,j\vee L})} \over =
-\nu_{j\vee L} [U_k,U_K] = 0
\\
\end {array}
$$
la nullité du déterminant de droite provenant du fait que ses deux premières colonnes sont égales.
En fin de compte, on obtient
$$
\nu_K^2\det\nolimits_{I' \times J'}(U) \buildrel{\rm def}\over=
\nu_K^2 [U_j,U_J] \buildrel {(\star_{j, J})} \over = \nu_K\nu_J [U_j,U_K] = 
\nu_J \times \nu_K [U_j,U_K] = 0
$$
\smallskip
ii)
Puisque l'égalité $\BW^r(u) = \Theta\,\nu$ est du type
$\hbox {matrice} = \hbox {vecteur-colonne} \times \hbox {vecteur-ligne}$, on
en déduit l'égalité d'idéaux $\calD_r(u) = \rmc(\Theta)\rmc(\nu)$.

\smallskip
iii)
Les idéaux $\rmc(\Theta)$ et $\rmc(\nu)$ étant fidèles, il en est de même de
leur produit $\calD_r(u)$.
\end {proof}

\noindent
Voici comment s'interprète la factorisation de $\BW^r(u)$ en terme
du sous-module $\DVect_r(u) \subset \BW^c(F)^\star$
défini en~\ref{DefDVect}.

\begin{prop}\label{FactorisationEtPgcdFort}
Soit $u : E \to F$ une application linéaire entre modules libres
et un isomorphisme déterminantal $\sharp
: \BW^r(F) \buildrel\simeq\over\longrightarrow \BW^c(F)^\star$ où $r,c$
sont complémentaires à $\dim F$.

\medskip
\noindent
Les deux assertions suivantes sont équivalentes.

\smallskip
\noindent
(i)
Il existe une factorisation $\bigwedge^r(u) = \Theta \, \nu$ où
$\nu \in \BW^r(E)^\star$ et $\Theta \in \bigwedge^r(F)$
avec $\Gr(\Theta) \geqslant 1$.

\smallskip
\noindent
(ii) Il existe une factorisation $\DVect_r(u) = \mu\,\fb$ où $\fb$ un idéal de type fini
et $\mu \in \BW^c(F)^\star$ avec $\Gr(\mu) \geqslant 1$.

\medskip
\noindent
La correspondance entre les deux est $\Theta^\sharp\leftrightarrow\mu$ et
$\Im\nu \leftrightarrow \fb$. En conséquence, $\DVect_r(u)$ admet un pgcd
fort (à savoir $\mu$)
si et seulement si $\Gr(\nu) \geqslant 2$.
\end{prop}

\begin{proof}\leavevmode
Soit $(e_j)_j$ une base de $E$.

\smallskip
\noindent
$(i) \Rightarrow (ii)$.
Appliquons l'égalité $\bigwedge^r(u) = \Theta \, \nu$ en~$e_J$ où
$\#J = r$ et donnons un coup de dièse :
$\big(\bigwedge^r(u)(e_J)\big)^\sharp = \nu(e_J)\,\Theta^\sharp$.  En
notant $\fb = \Im\nu$, idéal engendré par les $\nu(e_J)$, on a
$\DVect_r(u) = \Theta^\sharp\,\fb$.

\smallskip
\noindent
$(ii) \Rightarrow (i)$.
Il s'agit d'exhiber une factorisation de $\bigwedge^r(u)$ à partir de l'information
$\DVect_r(u) = \mu\,\fb$ où $\Gr(\mu) \geqslant 1$.
Pour $\#J = r$, on a
$\big(\bigwedge^r(u)(e_J)\big)^\sharp = \mu b_J$ pour un certain
$b_J \in \fb$ bien défini puisque $\mu$ est sans torsion.  Soit $\nu
: \bigwedge^r(E) \to \bfA$ la forme linéaire définie par $\nu(e_J) =
b_J$ et $\Theta \in \bigwedge^r(F)$ défini par $\Theta^\sharp = \mu$.
On a alors l'égalité $\bigwedge^r(u) = \Theta\,\nu$ comme on le voit
en évaluant en $e_J$ et en appliquant $\sharp$.
\end{proof}

\subsection {Déterminant de Cayley d'un complexe de Cayley}
\label{subsectionCayleyMacRae}

\index{de@déterminant de Cayley d'un complexe}%

Soit $(F_\sbullet, u_\sbullet)$ un complexe de Cayley.  Le principe
même de l'énoncé du théorème~\ref{Factorisation} fait que les
factorisations des puissances extérieures de deux différentielles
successives sont reliées entre elles.  Si on met de côté cette
inter-dépendance et si on oublie les isomorphismes
déterminantaux, on en déduit que chaque
différentielle, de manière indépendante, se factorise dans le sens
suivant :
$$
\textstyle \bigwedge^{r_k}(u_k) = \Theta_k \, \nu_k, \qquad 
\text{où $\Theta_k \in \bigwedge^{r_k}(F_{k-1})$ et 
$\nu_k \in \bigwedge^{r_k}(F_k)^\star$}
$$
avec les informations de profondeur suivantes : du côté de la forme
linéaire, on a $\Gr(\nu_k) \geqslant 2$ pour tout $k \geqslant 1$ tandis
que $\Gr(\Theta_k) \geqslant 2$ pour tout $k \geqslant 2$ mais
$\Gr(\Theta_1) \geqslant 1$.

\smallskip
Ces précisions de profondeur pour $\Theta_k$ et $\nu_k$ permettent
d'en déduire que la factorisation de $\bigwedge^{r_k}(u_k)$ est unique
à un inversible près au sens où toute autre factorisation est de la
forme $(\varepsilon \Theta_k, \varepsilon^{-1}\nu_k)$ pour un inversible
$\varepsilon$ comme on le voit en utilisant~\ref{uniciteFactorisation}.
Ainsi, le module engendré par la forme $c$-linéaire $\Theta_1^\sharp$ est bien défini.
On peut donc définir, sans ambiguïté, un sous-module libre de 
rang $1$ attaché à $F_{\sbullet}$ de la manière suivante.

\begin{theo}[Déterminant de Cayley vectoriel] 
\label{CayleyDetVectoriel}
\leavevmode

Soit $(F_\sbullet,u_\sbullet)$ un complexe de Cayley de caractéristique 
d'Euler-Poincaré $c$.
Le déterminant de Cayley vectoriel 
$\CayleyVect(F_\sbullet)$ de $F_\sbullet$ est
le sous-module libre de rang $1$ de $\bigwedge^{c}(F_0)^\star$ caractérisé par :
$$
\textstyle
\CayleyVect(F_\sbullet) \ = \ \bfA \mu \ \subset \ \bigwedge\nolimits^{c}(F_0)^\star
\qquad \qquad
\begin {array} {l}
\text{où $\mu$ est la forme $c$-linéaire
définie en~\ref{cLinearFormOfComplex}}
\\
\end {array}
$$
Cette forme $c$-linéaire $\mu$ 
est le pgcd fort du sous-module
$\DVect_{r_1}(u_1) \subset \BW^c(F_0)^\star$ (défini en 
\ref{DefDVect}).

\smallskip
Par abus de langage, on parlera de \og la \fg{} forme 
$c$-linéaire alternée de $F_\sbullet$ pour parler d'un générateur du
sous-module $\CayleyVect(F_\sbullet)$ de $\bigwedge^{c}(F_0)^\star$.
\end{theo}

\label{NOTA09-CayleyVect}%

\begin{proof}
Pour l'aspect pgcd fort, 
cela résulte du fait que la première 
différentielle de $F_\sbullet$ se factorise sous la forme 
$\bigwedge^{r_1}(u_1) = \Theta_1 \nu_1$ 
avec $\Gr(\nu_1) \geqslant 2$.
Et la proposition~\ref{FactorisationEtPgcdFort} permet de conclure.
\end{proof}

\medskip

Pour le cas très particulier de la caractéristique d'Euler-Poincaré $c = 0$, 
les auteurs préfèrent parler d'un \textit{scalaire} 
plutôt que de la \textit{forme linéaire} \og multiplication par ce scalaire\fg{}.
Le vocabulaire employé est alors \og déterminant de Cayley \fg{}.
C'est la raison pour laquelle nous avons repris cette terminologie 
en accolant l'adjectif \og vectoriel \fg{} dans le 
cas d'une caractéristique d'Euler-Poincaré quelconque, 
et nous parlerons de la \og forme déterminant de Cayley \fg{} 
(sous-entendu, forme $c$-linéaire alternée).

\medskip

Dans le cas où $c = 0$, à $F_\sbullet$ est donc attaché
un \textit{scalaire} régulier, défini à un inversible près: c'est le
pgcd fort de l'idéal $\calD_{r_1}(u_1)$ des mineurs pleins de $u_1$.
Dans ce cas, on préfère noter $\fC(F_\sbullet)$ \textit{l'idéal} de
$\bfA$ engendré par le scalaire régulier.

\medskip

L'autre cas important est celui de la caractéristique d'Euler-Poincaré
$c = 1$. \`A un tel complexe de Cayley  est attachée une forme linéaire $\mu : F_0 \to \bfA$
sans torsion, définie à un inversible près, 
qui est le pgcd fort des formes linéaires du sous-module $\DVect_{r_1}(u_1)$ de $F_0^\star$.

\bigskip
Si l'on veut exprimer la forme $c$-linéaire alternée, il faut encore
travailler.  La notion de complexe décomposé, développée dans le
chapitre~\ref{ComplexeDecompose}, participe à cette opération ; non
seulement d'une manière  pratique (elle permet d'assurer la
calculabilité de la factorisation) mais également de manière théorique
en aboutissant aux expressions de type quotients alternés.

\subsubsection*{Où l'on retrouve les modules de MacRae}

Rappelons qu'un module de MacRae de rang $c$ est un module possédant
une présentation $u : E \to F$ telle qu'en posant $r = \dim F -c$, on
ait $\calD_{r+1}(u) = 0$ et l'existence d'un pgcd fort de
$\DVect_r(u)$.  C'est une définition indépendante de la présentation
choisie (cf. le résultat~\ref{cFittingVectoriel} et la
définition plus intrinsèque~\ref{DefModuleMacRae} d'un module
de MacRae, équivalente à celle que nous venons de rappeler).

\begin{theo}[Module librement résoluble versus module de MacRae] \label{LibrementResolubleImpliqueMacRae}
Un module librement résoluble est un module de MacRae.

\smallskip
Précisément, supposons $M$ librement résolu par un complexe libre
$(F_\sbullet, u_\sbullet)$ de caractéristique d'Euler-Poincaré $c$.  Alors $M$ est
un module de MacRae de rang $c$ dont l'invariant de MacRae \mbox{$\vartheta
: \BW^c(M) \to \bfA$} s'obtient par passage au quotient du déterminant de
Cayley $\mu : \bigwedge^{c}(F_0) \to \bfA$ de $F_\sbullet$.

\smallskip
Avec un léger abus de langage, le déterminant de Cayley de
$F_\sbullet$ et l'invariant de MacRae de $M$ coïncident:
$$
\CayleyVect(F_\sbullet) \ \buildrel {\rm abus} \over = \ 
\MacRaeVect(M) 
$$
Note: à gauche, il s'agit d'un sous-module de $\bigwedge^{c}(F_0)^\star$ et à droite
d'un sous-module de $\bigwedge^{c}(M)^\star$.
\end{theo}

\begin{proof}
Le module $M$ est présenté par $u_1$ qui vérifie $\calD_{r_1 + 1}(u_1)
= 0$
(d'après~\ref{StableRank} ou encore~\ref{StableRankFromFactorization}).

De plus, par définition, la forme $c$-linéaire alternée $\mu$
est le pgcd fort de $\DVect_{r_1}(u_1)$.  Ainsi $M$ est  un module 
de MacRae de rang $c$. Le reste résulte de l'étude précédente.
\end{proof}

\subsection{L'isomorphisme $\rho : \Det(F_\sbullet)^\star \simeq \CayleyVect(F_\sbullet)$}

Ici, pour toute orientation $\bfe$ d'un module libre $F$, on choisit comme isomorphisme
déterminantal celui pour lequel le joker est à droite:
$$
\sharp_{\bfe} : 
\begin{array}[t]{rcl}
\bigwedge^{r}(F) & \longrightarrow & \bigwedge^{r'}(F)^\star \\ [0.4em]
\Theta & \longmapsto & \Theta^\sharp := 
\oriented{\Theta\wedge\sbullet\,}_\bfe
\end{array}
\qquad 
\text{avec $r + r' = \dim F$}
$$
Pour un inversible $\varepsilon$, on va utiliser l'égalité suivante
$[\Theta\wedge\sbullet\,]_{\varepsilon\bfe} \,= \, \varepsilon^{-1} [\Theta\wedge\sbullet\,]_{\bfe}$. 

\index{isomorphisme déterminantal}%

\begin{prop}[Dépendance en les orientations] \label{DependanceOrientation}

Soit $F_\sbullet$ un complexe de Cayley de caractéristique d'Euler-Poincaré $c$
et deux systèmes d'orientation $(\bfe_k)_{k \geqslant 0}$, $(\bfe'_k)_{k \geqslant 0}$ sur
$(F_k)_{k \geqslant 0}$. 
On note $\varepsilon_k \in \bfA$ l'unique scalaire inversible tel que $\bfe'_k = \varepsilon_k \bfe_k$.

En désignant par $(\Theta_k)_{k \geqslant 1}$
(resp. $(\Theta'_k)_{k \geqslant 1}$) le système de factorisation
associé au système d'orientation $(\bfe_k)_{k \geqslant 1}$ 
(resp. $(\bfe'_k)_{k \geqslant 1}$)
et par $\mu = \sharp_{\bfe_0}(\Theta_1) = [\Theta_1 \wedge \sbullet]_{\bfe_0}$  
(idem avec un prime), 
on a :
\begin{itemize}
\item dans $\bigwedge^{r_k}(F_{k-1})$, l'égalité :
$$
\textstyle
\Theta'_k 
\ = \ 
\dfrac{\varepsilon_k \,\varepsilon_{k+2} \, \cdots}
{\varepsilon_{k+1} \, \varepsilon_{k+3} \, \cdots}
\ \Theta_k 
$$
\item 
dans $\bigwedge^{c}(F_{0})^\star$, l'égalité :
$$
\mu' 
\ = \ 
\dfrac{\varepsilon_1 \,\varepsilon_{3} \, \cdots}
{\varepsilon_{0} \, \varepsilon_{2} \, \cdots}
\ \mu
$$ 
\end{itemize}
\end{prop}

\begin{proof}
\leavevmode

\begin{itemize}
\item 
Par hypothèse, on a les égalités:
$$
\textstyle
\bigwedge^{r_k}(u_k) \ =\  \Theta_k\ [\Theta_{k+1} \wedge\sbullet\,]_{\bfe_k} \ =\ 
\Theta'_k\ [\Theta'_{k+1} \wedge\sbullet\,]_{\bfe'_k}
\leqno (\star)
$$
Supposons $\Theta'_{k+1} = \lambda_{k+1}\, \Theta_{k+1}$ avec $\lambda_{k+1}$ inversible.
C'est le cas pour $k =n$ (prendre $\lambda_{n+1} = 1$).
Il vient :
$$
[\Theta'_{k+1} \wedge\sbullet\,]_{\bfe'_k} \ =\ 
\lambda_{k+1} [\Theta_{k+1} \wedge\sbullet\,]_{\bfe'_k} \ = \ 
\lambda_{k+1} \,\varepsilon_k^{-1}\,[\Theta_{k+1} \wedge\sbullet\,]_{\bfe_k}
$$
la deuxième égalité provenant de 
$[\Theta\wedge\sbullet\,]_{\varepsilon\bfe} \,= \, \varepsilon^{-1} [\Theta\wedge\sbullet\,]_{\bfe}$. 
En reportant cela dans l'égalité droite de $(\star)$, on obtient :
$$
\Theta_k\ [\Theta_{k+1} \wedge\sbullet\,]_{\bfe_k} \ =\ 
\lambda_{k+1} \, \varepsilon_k^{-1}\,
\Theta'_k\ [\Theta_{k+1} \wedge\sbullet\,]_{\bfe_k}
$$
En simplifiant par $[\Theta_{k+1} \wedge\sbullet\,]_{\bfe_k}$ (qui est une forme linéaire 
sans torsion),
on en déduit :
$$
\Theta_k \ =\  \lambda_{k+1} \,\varepsilon_k^{-1}\, \Theta'_k
\qquad \idest{} \qquad
\Theta'_k \ = \  \lambda_k \Theta_k
\quad \hbox {où $\lambda_k$ est défini par } 
\lambda_k\,\lambda_{k+1} = \varepsilon_k
$$

\item 
On a $\Theta'_1 =  \lambda_1 \Theta_1$ 
et $\bfe'_0 = \varepsilon_0 \bfe_0$.
Montrons que $\mu' = \lambda_1\dfrac{1}{\varepsilon_0 } \mu$, ce qui suffira à conclure.

On a les égalités suivantes :
$$
\mu' 
\ = \  
[\Theta'_1 \wedge \sbullet]_{\bfe'_0}
\ = \ 
[\lambda_1 \Theta_1 \wedge \sbullet]_{\varepsilon_0 \bfe_0}
\ = \ 
\lambda_1 [\Theta_1 \wedge \sbullet]_{\varepsilon_0 \bfe_0}
\ = \
\lambda_1 \dfrac{1}{\varepsilon_0 }[\Theta_1 \wedge \sbullet]_{\bfe_0}
\ = \ 
\lambda_1\dfrac{1}{\varepsilon_0 } \mu
$$
\end{itemize}
\end{proof}

\subsubsection{Une autre vision du déterminant de Cayley $\CayleyVect(F_\sbullet)$}

\`A un complexe de Cayley $F_\sbullet$,
on associe le $\bfA$-module libre de rang $1$ suivant 
qui ne dépend que des termes $F_k$ du complexe :
$$
\textstyle
\Det(F_\sbullet) \ = \ 
\bigotimes\limits_{k \, \rm pair} \bigwedge^{\dim  F_k}(F_k)
\ \otimes \ \!\!
\bigotimes\limits_{h \, \rm impair} \!\!
\bigwedge^{\dim  F_h}(F_h)^\star
\, \ = \ \, 
\bigwedge^{\dim  F_0}(F_0) \, \otimes \, \bigwedge^{\dim  F_1}(F_1)^\star \, \otimes \, 
\bigwedge^{\dim  F_2}(F_2) \, \otimes \, \cdots
$$

\begin{prop}
Soit $\bfe_k$ une orientation de $F_k$ pour $k \geqslant 0$ et 
$\bfe_k^\star$ l'orientation déduite sur $F_k ^\star$.
Considérons la forme $c$-linéaire $\mu$, définie en~\ref{cLinearFormOfComplex}, 
dépendant des orientations $(\bfe_k)$.

On dispose alors d'un morphisme $\rho$ bien défini (\idest{} indépendant des orientations choisies) :
$$
\textstyle
\rho : \ \Det(F_\sbullet)^\star \ \longrightarrow \ \bigwedge^{c}(F_0)^\star, 
\qquad \quad
\beta_0 \otimes \alpha_1 \otimes \beta_2 \cdots 
\ \longmapsto \ 
\beta_0(\bfe_0) \, \bfe_1^\star(\alpha_1) \, \beta_2(\bfe_2) \cdots \ \, 
\times 
\mu 
$$
Ce morphisme est injectif et établit un isomorphisme de $\Det(F_\sbullet)^\star$ sur $\CayleyVect(F_\sbullet)$:
$$
\textstyle
\CayleyVect(F_\sbullet) \ = \ 
\rho\big(\Det(F_\sbullet)^\star \big) \ \subset \ 
\bigwedge^{c}(F_0)^\star
$$
\end{prop}

\label{NOTA09-Det}%
%
%

\begin{proof}
Montrons que $\rho$ est indépendant des orientations.
Soit $\bfe'_k= \varepsilon_k \, \bfe_k$ une autre orientation pour~$F_k$ 
(où $\varepsilon_k$ est inversible), de sorte 
qu'au niveau des formes linéaires, on a ${\bfe'_k}^\star = 
\varepsilon_k^{-1} \bfe_k^\star$.
On a donc les égalités :
$$
\beta_k(\bfe'_k) \, = \, \varepsilon_k \, \beta_k(\bfe_k)
\qquad \text{et} \qquad 
{\bfe'_h}^\star(\alpha_h) \, = \, \varepsilon_h^{-1} \, \bfe_h^\star(\alpha_h)
$$
Ainsi 
$$
\beta_0(\bfe'_0) \, {\bfe'_1}^\star(\alpha_1) \, \beta_2(\bfe'_2) \cdots \ \, 
\ = \ 
\dfrac{\varepsilon_0 \,\varepsilon_2 \, \cdots}
{\varepsilon_1 \, \varepsilon_3 \, \cdots} \ \, 
\beta_0(\bfe_0) \, \bfe_1^\star(\alpha_1) \, \beta_2(\bfe_2) \cdots \ \, 
$$
En multipliant cette égalité par $\mu'$ 
associée à $(\bfe'_k)$,
qui vaut 
$
\mu' 
=
\dfrac{\varepsilon_1 \,\varepsilon_{3} \, \cdots}
{\varepsilon_{0} \, \varepsilon_{2} \, \cdots}
\ \mu
$ 
(d'après~\ref{DependanceOrientation}), 
on obtient :
$$
\beta_0(\bfe'_0) \, {\bfe'_1}^\star(\alpha_1) \, \beta_2(\bfe'_2) \cdots \ \, \mu'
\ = \ 
\beta_0(\bfe_0) \, \bfe_1^\star(\alpha_1) \, \beta_2(\bfe_2) \cdots \ \, 
\mu
$$

\medskip
\noindent
Pour l'injectivité de $\rho$, il suffit d'utiliser le fait que 
$\mu$ est sans torsion, par définition (cf.~\ref{cLinearFormOfComplex}).
Concernant l'image de $\rho$:
par définition, $\CayleyVect(F_\sbullet)$ est 
le $\bfA$-module libre de rang $1$ de base $\mu$, tout comme $\Im \rho$.
\end{proof}

\subsection{Bonus : idéaux de factorisation d'un complexe exact}

La donnée ici est un complexe libre \emph{exact} $(F_\sbullet, u_\sbullet)$
dont nous notons comme auparavant $(r_k)_{0\leqslant k\leqslant n+1}$ la suite des
rangs attendus et $(\Theta_k)_{k \geqslant 1}$ un système de factorisation:
$\bigwedge^{r_k}(u_k) = \Theta_k\,\Theta_{k+1}^\sharp$, cf le théorème~\ref{Factorisation}.
Pour alléger, nous posons $\fD_k = \calD_{r_k}(u_k)$ et utiliserons parfois
$\nu_k := \Theta_{k+1}^\sharp \in \BW^{r_k}(F_{k})^\star$.

\label{NOTA09-fDk}%

\medskip
Nous allons considérer l'idéal contenu $\rmc(\Theta_k)$ du vecteur $\Theta_k$.
Puisque l'égalité $\BW^{r_k}(u_k) = \Theta_k \nu_k$ est du type
$\hbox {matrice} = \hbox {vecteur-colonne} \times \hbox {vecteur-ligne}$, on
en déduit la factorisation $\fD_k = \rmc(\Theta_k) \rmc(\Theta_{k+1})$.
Il faut noter que cette information obtenue au niveau des idéaux
est plus faible que celle fournie par la factorisation \og vectorielle\fg{} de 
$\BW^{r_k}(u_k)$.

\begin{defn}[Idéaux de factorisation d'un complexe exact] \label{IdeauxFactorisation}
Dans le contexte ci-dessus, on définit les idéaux $(\fB_k)_{1 \leqslant k \leqslant n+1}$
où $\fB_k = \rmc(\Theta_k)$ est l'idéal engendré par les coordonnées de $\Theta_k$.
\end{defn}

\label{NOTA09-fBk}%
%
%

\begin{prop}[Idéaux déterminantiels versus idéaux de factorisation]
Les idéaux déterminantiels $\fD_k$ de $F_\sbullet$ 
se factorisent à l'aide des idéaux $\fB_k$ sous la forme :
$$
\fD_k = \fB_k\fB_{k+1}
\qquad \text{avec } \qquad 
\left\{
\begin{array}{l}
\Gr(\fB_1) \geqslant 1\\
\Gr(\fB_k) \geqslant k \quad \text{pour tout $k \geqslant 2$} \\ 
\end{array}
\right.
\qquad \text{ et } \qquad 
\fB_{n+1} = \bfA
$$
En particulier, pour un complexe de caractéristique d'Euler-Poincaré $0$, 
l'idéal $\fB_1$ est monogène.
\end{prop}

\begin{proof} \leavevmode

\noindent
Une façon équivalente d'obtenir la factorisation de $\fD_k$ consiste à
évaluer l'égalité $\bigwedge^{r_k}(u_k)
= \Theta_k\,\Theta_{k+1}^\sharp$ en une base de
$\bigwedge^{r_{k}}(F_k)$,  ce qui fournit une égalité de vecteurs de
$\bigwedge^{r_{k}}(F_{k-1})$.  En prenant les coordonnées de ces
vecteurs dans une base, on obtient la factorisation annoncée.

\medskip
\noindent
On a l'inclusion $\fD_k \subset \fB_k$, conséquence de l'égalité
$\fD_k = \fB_k\fB_{k+1}$, donc $\Gr(\fB_k) \geqslant \Gr(\fD_k)$. On conclut
en utilisant ``What makes a complex exact'' (cf. \ref{WhatMakesAComplexExact}).

\medskip
\noindent
Si $r_0 = 0$, alors $r_1 = \dim F_0$, et 
ainsi le vecteur $\Theta_1 \in \bigwedge^{r_1}(F_0) \simeq \bfA$ 
s'identifie à un scalaire, si bien que~$\fB_1$ est monogène.
\end{proof}

\begin {rmq}
Si l'on considère un module librement résoluble $M$, nous ignorons
si les idéaux de factorisation sont indépendants du choix de la résolution de $M$.
Northcott passe totalement sous silence cet aspect 
(confer la phrase juste après la définition page $218$ : 
\og Sometimes we shall simply say that $(\fB_k)_k$ 
is a system of factorization ideals for $M$ without specifying the resolution 
which produces them \fg{}).
\end {rmq}

\medskip

Voici maintenant un théorème assez subtil.
D'ailleurs, Northcott écrit juste avant son énoncé : \og the next result is more unexpected \fg{} 
(cf. \cite[théorème 8, chapitre 7]{NorthcottFFR}).

\begin{theo} \label{RacineIdealFactorisation}
On conserve le contexte ci-dessus d'un complexe exact $F_\sbullet$.

\medskip
i) Pour $k \geqslant 2$, on a l'égalité $\sqrt {\fD_k} = \sqrt {\fB_k}$.

ii) On dispose des inclusions 
$$
\sqrt{\fD_1} \ \subset \  \sqrt{\fD_2}
\subset \ \cdots \ \subset \ \sqrt{\fD_k} \subset \ \cdots
$$
\end{theo}

Pour la preuve de ce théorème, nous avons mis 
au point deux énoncés \ref{FactorisationKerFacteurDirect} et \ref{FactorisationImFacteurDirect} 
concernant des propriétés relatives à la factorisation 
$\bigwedge^r(u) = \Theta \, \nu$.
De ces deux énoncés, on peut essentiellement retenir cela : 
\og si la forme linéaire $\nu$ est surjective, alors le noyau de $u$ est facteur direct \fg{} et 
\og si l'image de $u$ est facteur direct, alors la forme linéaire $\nu$ est surjective \fg{}.
Plus précisément :
\begin{quote}
{\it 
Soit $u : E \to F$ avec une factorisation $\bigwedge^r(u) = \Theta\, \nu$ telle que 
$\Gr(\Theta) \geqslant 1$ et $\Gr(\nu) \geqslant 1$.
$$
\Im u \text{ facteur direct dans } F 
\ \implies \
\nu \text{ surjective}
\ \implies \
\Ker u \text{ facteur direct dans } E 
\leqno (\star)
$$
}
\end{quote}

\begin{proof}[Preuve du théorème \ref{RacineIdealFactorisation}] \leavevmode

\noindent
i) Il suffit de montrer les inclusions $\fD_k \subset \fB_k \subset \sqrt{\fD_k}$.
La première résulte de $\fD_k = \fB_k\fB_{k+1}$.
Quant à la seconde, il suffit de montrer que $\fB_k \subset \sqrt{\fB_{k+1}}$.
Après localisation en $a \in \fB_k$, cela revient à montrer l'implication
$1 \in \fB_k \Rightarrow 1 \in \fB_{k+1}$.

\smallskip
\noindent
Supposons donc $1 \in \fB_k$. Puisque $\fB_k = \rmc(\Theta_k) = \rmc(\nu_{k-1})$, la forme linéaire 
$\nu_{k-1}$ est surjective.
En appliquant $(\star)$ à $u_{k-1}$, on en déduit que $\Ker u_{k-1}$ est facteur direct.
Mais $\Ker u_{k-1} = \Im u_k$ et de nouveau $(\star)$ appliqué à $u_k$ fournit
la surjectivité de $\nu_k$ c'est-à-dire $1 \in \fB_{k+1}$.

\medskip
\noindent
ii) Le fait que la suite $(\sqrt{\fD_k})_{k \geqslant 1}$ soit croissante résulte de :
$$
\fD_{k-1} \ \subset \ \fB_k \ \subset \ \sqrt{\fB_k} \, = \, \sqrt{\fD_k}
$$
\end{proof}

\begin {rmq} 
La contrainte $k \geqslant 2$ dans l'égalité $\sqrt {\fD_k} = \sqrt
{\fB_k}$ est indispensable.  Ou encore, ce n'est pas vrai que $\nu_0
:= \Theta_1^\sharp$ surjective entraîne $\nu_1 := \Theta_2^\sharp$
surjective.  Considérons une forme linéaire $u : E \to \bfA$ sans
torsion. Avec $\Theta = 1$, $\nu = u$ et $r=1$, on a la factorisation
$\bigwedge^r(u) = \Theta \, \nu$.
Ainsi, pour obtenir un contre-exemple, il
suffit de trouver un complexe exact dont la première différentielle
est une forme linéaire \emph{non} surjective.  Ceci n'est pas
difficile : prendre le complexe de Koszul d'une suite $\ua =
(a_1, \ldots, a_n)$ complètement sécante mais non unimodulaire.
\end {rmq}

\subsection{Petit historique}

La notion de structure multiplicative joue un rôle capital dans notre approche
sur le résultant et il nous semble donc important d'apporter quelques
précisions de nature historique. Nous avons retenu cette terminologie 
\og structure multiplicative\fg{} car elle 
s'inspire directement du titre du chapitre~7 
(The multiplicative structure) de l'ouvrage de Northcott, mais il n'est pas
certain qu'elle évoque un processus de factorisation des puissances
extérieures ad hoc des différentielles d'une résolution libre.

Les initiateurs sont Buchsbaum~\&~Eisenbud dans ``Some structure theorems for
finite free resolutions'' (1974), en particulier leur théorème~3.1, déja
annoncé par ces auteurs en 1972. On retrouve l'énoncé de ce théorème dans
le papier ``On the Buchsbaum-Eisenbud theory of finite free resolutions'' 
de Eagon~\&~Northcott paru en janvier 1973 (c'est très exactement leur théorème~3). 
Les dates 1973-1974 de parution s'expliquent par le fait que Eagon~\&~Northcott 
ont eu accès au preprint des précurseurs
Buchsbaum~\&~Eisenbud. On peut dire que grosso modo, la volonté de 
Eagon~\&~Northcott était de faire simple (ce qui est est également le cas l'ouvrage de
Northcott paru en 1976), comme en témoignent les quelques passages suivants de
leur papier~:
\begin{quote}
{\it
``The techniques used are all relativement elementary'',
``... makes the subject more accessible to the newcomer''. 
``It was primarily because Theorem~3 has such a rich variety of applications, 
that we have attempted to find a simpler treatment. What has gone before, we
would claim, uses only the simplest results of Commutative Algebra. What comes
now, it will be found, requires no previous knowledge of Homological Algebra.''
}
\end{quote}

\noindent
Melvin Hochster, à la fin de son review (1975) de l'article ``Some structure
theorems for finite free resolution'' de Buchsbaum \& Eisenbud, dit lui-même~:

\begin{quote}
{\it 
``Finally, we note that J. A. Eagon and D. G. Northcott have given an account
of some of the authors' results without the use of exterior algebra or
homological methods. In particular, the commutative diagrams of tensor
products of several exterior algebra that abound in the Buchsbaum-Eisenbud
paper are absent, and some readers may prefer the Eagon-Northcott treatment
first.''
}
\end{quote}

\noindent
Et dans son review (1978) du livre de Northcott~:

\begin{quote}
{\it 
``There are two features in which the author's treatment
differs from existing accounts of the subject: first, he confines himself
almost entirely to elementary methods, avoiding Ext, Tor, and even exterior
powers (we shall do likewise), and, second, he exploits a new notion of grade
(or depth) in the non-Noetherian case which permits him to dispense entirely
with the Noetherian restrictions on the ring. The very elementary form of the
treatment enables the author to make accessible some fancy results from the
homological theory of rings to readers with virtually no background in
algebra.''
}
\end{quote}

\noindent
De manière plus récente (2011), le fondement de la théorie des résolutions
libres est revisité dans l'article~\cite{CoquandQuitte} 
de Coquand~\&~Quitté, l'objectif étant d'une
part de rendre encore plus élémentaire l'approche de Northcott et d'autre part
de supprimer l'utilisation du spectre premier. 
Dans cet article, l'existence de la structure multiplicative repose sur le théorème de
proportionnalité, ce qui permet de remplacer le contexte \og résolution libre
finie/complexe exact \fg{} par \og complexe de Cayley \fg{} ; dans le premier contexte, l'escalier de
minoration des profondeurs est $1,2,3,4, \dots$
à comparer avec l'escalier $1,2,2,2,\dots$ pour un complexe de Cayley.
Signalons que les théorèmes de proportionnalité qui figurent ou bien dans la
proposition~7 du papier de Eagon~\&~Northcott ou bien dans le lemme~2 du 
chapitre~7 de Northcott sont plus faibles que le lemme~8.2 de~\cite{CoquandQuitte}.
Il nous est ainsi difficile de comprendre les propos de Hochster dans les dernières lignes de 
son review sur le papier de Buchsbaum~\&~Eisenbud~: 

\begin{quote}
{\it 
``The arguments of Eagon and Northcott show that \dots{} in the non-Noetherian
case and that for (3.1), it is not necessary to assume the complex acyclic:
the depths need only climb $1,2,2,\ldots,2,\ldots$ instead of
$1,2,3,\ldots,i,\ldots$ as in the acyclic case.''
}
\end{quote}

\cleardoublepage

\section{Invariants de MacRae (scalaires et vectoriels) attachés à $\protect\uP$}
\label{ChapMacRaeForP}

Après avoir étudié les invariants de MacRae associés au jeu étalon
généralisé (cf.~\ref{ChapJeuEtalonGeneralise}), nous les étudions
ici dans le cadre d'une suite $\uP$ générale.  Nous allons \og
jongler\fg{} entre les cas $\uP$ générique, $\uP$ régulière et $\uP$
quelconque, le cadre $\uP$ régulière permettant d'assurer que les
modules qui interviennent sont de MacRae.  Dans un premier temps, nous
allons utiliser le cadre générique de façon à pouvoir spécifier de
manière précise (sans ambiguïté de signe) des générateurs de ces
invariants.

\medskip

Pour réaliser ce programme, nous allons avoir besoin des deux résultats fondamentaux 
(qui nécessitent tous les deux le théorème de structure multiplicative) :
\begin{itemize}
\item 
l'invariant de MacRae se spécialise, cf.~\ref{ExtensionScalaires}
\item
un module librement résoluble est un module de MacRae, cf.~\ref{LibrementResolubleImpliqueMacRae}
\end{itemize}

\medskip

Interviennent les acteurs habituels: le $\bfA[\uX]$-module gradué
$\bfB= \bfA[\uX]/\langle \uP \rangle$, sa composante homogène~$\bfB_d$
de degré $d$, un déterminant bezoutien~$\nabla$ non unique,
déterminant d'une matrice homogène~$\dsV$ exprimant la suite~$\uP$ en
fonction de la suite~$\uX$ (confer~\ref{DefNabla}).
Un tel déterminant bezoutien permet de définir un nouveau $\bfA[\uX]$-module indépendant du choix 
de la matrice~$\dsV$ (et de $\nabla$), à savoir $\bfB' = \bfA[\uX]/\langle \uP, \nabla \rangle$ 
(cf.~\ref{DefQuotientsP}).
L'entier $\widehat d_i$ désigne le produit $\prod_{j \neq i} d_j$
où $D = (d_1, \dots,d_n)$ est le format de degrés 
du système~$\uP$.

\medskip

D'après~\ref{ResolutionQuotients}, lorsque $\uP$ est régulière, 
les modules $\bfB_d$ et $\bfB'_\delta$ sont librement résolubles
a fortiori des modules de MacRae.
D'après le résultat de spécialisation, 
on peut se permettre de déterminer l'invariant de MacRae de ces modules
en supposant $\uP$ générique, puis ensuite spécialiser.

\medskip

Précisément, dans le premier énoncé qui vient, nous allons dégager un générateur
privilégié (et pas seulement défini au signe près) de chaque
module/idéal de MacRae pour la suite générique :

\begin{itemize}
\item pour $d \geqslant \delta + 1$, 
le module $\bfB_d$ est de MacRae de rang $0$ et 
son invariant de MacRae $\MacRae(\bfB_d)$ est engendré par le scalaire $\calR_d$ ;

\smallskip

\item 
le module $\bfB_\delta$ est de MacRae de rang $1$ et 
son invariant de MacRae $\MacRaeVect(\bfB_\delta)$ 
est engendré par la forme linéaire $\uomegares \in \bfB_\delta^\star$ ;

\smallskip

\item 
le module $\bfB'_\delta$ est de MacRae de rang $0$ et 
son invariant de MacRae 
$\MacRae(\bfB'_\delta)$ est engendré par le scalaire $\omegares(\nabla)$.
\end{itemize}

\medskip
Signalons que le choix des générateurs invoqué ci-dessus est attaché à
la \emph {suite} $\uP$ et pas seulement à l'idéal $\langle\uP\rangle$.
Dans le chapitre suivant (cf le théorème~\ref{MacRaeEqualities}), nous montrerons,
pour les modules de MacRae de rang~0 ci-dessus, que les générateurs invariants
de MacRae que nous allons exhiber sont tous égaux,
ce qui permettra de définir le résultant de $\uP$.


\medskip

Dans la suite, nous utilisons, sans toujours le mentionner 
explicitement, la correspondance biunivoque naturelle entre les formes
linéaires sur $\bfB_\delta$ et celles sur $\bfA[\uX]_\delta$
nulles sur $\langle\uP\rangle_\delta \buildrel {\rm def} \over = \Im\Syl_\delta$.
Ce dernier sous-module de $\bfA[\uX]_\delta^\star$ se désigne de manière cryptique par la
notation $\Ker \transpose{\Syl_\delta}$; il  contient (en général strictement) 
le sous-module~$\DVect_{s_\delta}(\Syl_\delta)$ des formes linéaires déterminantales de
l'application de Sylvester~$\Syl_\delta$.

\subsection{Normalisation des générateurs de MacRae:
  $\MacRaeVect(\bfB_\delta)$, $\MacRae(\bfB'_\delta)$ et $\MacRae(\bfB_d)_{d \geqslant \delta+1}$}

Comme annoncé précédemment, ici, la suite $\uP = 
(P_1, \dots, P_n)$ de polynômes de $\bfA[\uX] = \bfA[X_1, \dots, X_n]$
est \emph {générique}.
On rappelle ce que cela signifie. L'anneau $\bfA$ est un anneau de polynômes 
sur un anneau~$\bfk$ quelconque et les indéterminées de~$\bfA$ sur $\bfk$ 
sont allouées aux coefficients de chaque $P_i$.
Pour $a \in \bfA$, l'expression \og $a$ est homogène en~$P_i$ \fg{} 
doit être comprise de la façon suivante :
$a$ vu comme élément de $\bfR_i[\text{coefficients de $P_i$}]$ est homogène,
où $\bfR_i = \bfk[\text{coefficients des $P_j$ pour $j \neq i$}]$.

\medskip
Dans l'anneau de polynômes $\bfA$, on va séparer les indéterminées en deux paquets : 
les indéterminées $p_1, \dots, p_n$, coefficients des~$X_i^{d_i}$ (intervenant pour
le jeu étalon généralisé),  et les autres $a_{i,\gamma}$ avec $|\gamma| = d_i$ et $X^\gamma \neq X_i^{d_i}$, 
de sorte que $\bfA = \bfk[p_1, \dots, p_n][(a_{i,\gamma})]$.

\begin{theo}[Définition/caractérisation des scalaires $\calR_d$ et de la forme $\omegares$] 
\label{PoidsNormalisationMacRae}
\leavevmode

Soit $\uP$ générique.
\begin{enumerate}[\rm i)]
\item 
En degré $d \geqslant \delta +1$, 
l'idéal de MacRae $\MacRae(\bfB_d)$ possède un unique générateur $\calR_d$
contenant le monôme $p_1^{\widehat d_1} \cdots p_n^{\widehat d_n}$. Ce générateur $\calR_d$
est homogène en $P_i$ de poids 
$\widehat d_i = \dim \Jex_{1\setminus 2,d}^{(i)}$.

En considérant $\bfA$ comme anneau de polynômes en les $a_{i,\gamma}$ 
à coefficients dans $\bfk[p_1, \dots, p_n]$, 
le monôme $p_1^{\widehat d_1} \cdots p_n^{\widehat d_n}$
est le \textrm{terme constant} de $\calR_d$.

En considérant cette fois $\bfA$ comme anneau de polynômes en les $p_1, \dots, p_n$
à coefficients dans $\bfk[(a_{i,\gamma})]$, 
ce monôme $p_1^{\widehat  d_1} \cdots p_n^{\widehat d_n}$ 
est la composante homogène dominante de $\calR_d$.

\item 
En degré $d = \delta$, le $\bfA$-module de MacRae $\MacRaeVect(\bfB_\delta)$ 
possède un unique générateur $\uomegares \in \bfB_\delta^\star$ 
tel que l'évaluation de
la forme linéaire $\omegares \in \bfA[\uX]_\delta^\star$ 
en le \MoutonNoir{}, à savoir 
$\omegares(X^\emouton) \in \bfA$, 
contienne le monôme $p_1^{\widehat d_1-1} \cdots p_n^{\widehat d_n-1}$.
De plus, chaque $\omegares(X^\alpha)$ pour $|\alpha| = \delta$ est homogène en $P_i$ de poids 
$\widehat d_i - 1
= \dim \Jex_{1\setminus 2,\delta}^{(i)}$.

\medskip
En particulier, 
l'idéal de MacRae $\MacRae(\bfB'_\delta)$ est engendré par $\omegares(\nabla)$, 
homogène en $P_i$ de poids $\widehat d_i$ 
et contient le monôme $p_1^{\widehat d_1} \cdots p_n^{\widehat d_n}$.
\end{enumerate}
\end{theo}

\index{invariant de MacRae}%
\index{forme linéaire!intrinsèque $\omegares=\omegaRes{\uP}$}%

\begin{proof}
i) Rappelons que $\MacRae(\bfB_d)$ est engendré par un élément régulier, pgcd fort 
des mineurs pleins de $\Syl_d(\uP)$ ; chacun de ces mineurs est homogène 
en $P_i$ de poids le nombre de colonnes en $P_i$.
Parmi ces mineurs, il en existe un bien explicite, à savoir $\det W_{1,d}$, 
qui est régulier.
Ainsi tout générateur de~$\MacRae(\bfB_d)$, en tant que diviseur de $\det W_{1,d}$, 
est homogène en chaque $P_i$. 
Nous allons prouver l'existence d'un générateur contenant le monôme 
$p_1^{\widehat d_1} \cdots p_n^{\widehat d_n}$; un tel 
générateur sera homogène en $P_i$ de poids~$\widehat d_i$, 
poids égal à la dimension de~$\Jex_{1\setminus 2,d}^{(i)}$ 
d'après~\ref{LemmeDenombrementJ1moins2id}.

\medskip

Plaçons-nous dans un premier temps au-dessus de $\bfk = \bbZ$, en considérant $\bfA = \bbZ[\indetsPi]$.
Les seuls inversibles de $\bfA$ étant $\pm 1$, 
l'idéal $\MacRae(\bfB_d)$ ne possède que deux générateurs.
Nous en prenons un quelconque, disons $\calR$ (l'autre étant $-\calR$).
Considérons la spécialisation 
$\bbZ[p_1, \dots, p_n][(a_{i,\gamma})] \to \bbZ[p_1, \dots, p_n]$ 
qui consiste à mettre à $0$ les indéterminées $a_{i,\gamma}$ :
c'est la spécialisation du jeu générique en le jeu étalon généralisé.
Comme l'invariant de MacRae se spécialise et que nous avons 
déterminé celui du jeu étalon généralisé,
cette spécialisation réalise $\calR \mapsto \pm \, p_1^{\widehat d_1} \cdots p_n^{\widehat d_n}$.
Quitte à changer $\calR$ en $-\calR$, on peut 
supposer que $\calR \mapsto p_1^{\widehat d_1} \cdots p_n^{\widehat d_n}$.
On a donc $\calR = p_1^{\widehat d_1} \cdots p_n^{\widehat d_n} + S$ 
où $S$ est dans l'idéal de $\bfA$ engendré par les~$a_{i,\gamma}$. 
Ainsi, $\calR$ contient le monôme~$p_1^{\widehat d_1} \cdots p_n^{\widehat d_n}$.

\medskip

Revenons à $\bfk$ avec la spécialisation $\bbZ \to \bfk$. 
Considérons l'image du générateur $\calR$ précédent. 
C'est un générateur de $\MacRae(\bfB_d)$ au-dessus de $\bfk$ (par propriété 
de l'invariant de MacRae) 
et il contient donc le monôme $p_1^{\widehat d_1} \cdots p_n^{\widehat d_n}$. 

\bigskip

ii) La preuve de ce point présente beaucoup d'analogies avec celle du
point précédent et de ce fait nous détaillons seulement la première étape où $\bfk = \bbZ$.  
Le $\bfA$-module $\MacRaeVect(\bfB_\delta)$
possède alors deux générateurs opposés l'un de l'autre : nous en choisissons un. 
Ce générateur est une forme linéaire $\overline \mu : \bfB_\delta \to \bfA$, pgcd fort 
des formes linéaires de $\DVect_{s_\delta}(\Syl_\delta)$ vues sur $\bfB_\delta$.
Travaillons au niveau de $\bfA[\uX]_\delta$. 
La forme linéaire $\omega$ est multiple de la forme linéaire $\mu$ : 
écrivons $\omega = a \mu$ où $a$ est un élément régulier de~$\bfA$.
Pour $X^\alpha$ de degré $\delta$, le scalaire $\omega(X^\alpha)$ est 
homogène en chaque $P_i$ d'un poids \textit{indépendant} de~$X^\alpha$: c'est par
exemple le poids en $P_i$ de $\omega(X^\emouton) = \det W_{1,\delta}$.
Grâce à l'égalité $\omega = a \mu$, nous en déduisons que les $\mu(X^\alpha)$ sont
homogènes en $P_i$, d'un poids \textit{indépendant} de $X^\alpha$.

\medskip

Nous procédons à la spécialisation $a_{i,\gamma} \mapsto 0$,
celle correspondant à la spécialisation du jeu générique~$\uP$ en le jeu étalon généralisé. 
D'après~\ref{MacRaeJeuEtalonDelta}, elle réalise donc 
$\mu \to \varepsilon\, p_1^{\widehat d_1 -1} \cdots p_n^{\widehat d_n -1} (X^\emouton)^\star$ avec
$\varepsilon = \pm 1$. Quitte à changer $\mu$ en $-\mu$, nous pouvons
supposer $\varepsilon = 1$.

Ainsi, $\mu(X^\emouton)$, vu comme un polynôme en les $(a_{i,\gamma})$
à coefficients dans $\bfk[p_1, \dots, p_n]$, a pour terme constant
$p_1^{\widehat d_1 -1} \cdots p_n^{\widehat d_n -1}$. Nous en déduisons
que $\mu(X^\emouton)$, vu dans $\bfA$, est homogène en chaque $P_i$
de poids $\widehat d_i -1$.
D'après la précision d'indépendance en
$X^\alpha$ formulée auparavant, on en déduit que chaque
$\mu(X^\alpha)$ est homogène en $P_i$ de même poids 
$\widehat d_i -1$
(poids égal à la dimension de $\Jex_{1\setminus 2, \delta}^{(i)}$ d'après~\ref{LemmeDenombrementJ1moins2id}).

\bigskip

La fin de l'énoncé 
\og l'idéal $\MacRae(\bfB'_\delta)$ est engendré par $\omegares(\nabla)$ \fg{} 
résulte du résultat d'algèbre commutative général~\ref{Rank1to0MacRae}. 
Les précisions sur le poids s'obtiennent via les deux points précédents.
\end{proof}

\label{NOTA10-omegares-Rd}%
%
%

\subsection{Les générateurs de MacRae pour une suite $\protect\uP$ quelconque}

Soit $\uP$ un système quelconque.  Le $\bfA$-module $\bfB_d$ est de
présentation finie, présenté par l'application de Sylvester $\Syl_d$,
ce qui permet de considérer ses idéaux de Fitting $\calF_i(\bfB_d)$.
En notant $c= \chi_d$ la caractéristique d'Euler-Poincaré du complexe
$\rmK_{\sbullet,d}$, on~a $\calF_{c-1}(\bfB_d) = 0$ permettant de
définir le sous-module de Fitting $\FittVect_c(\bfB_d) \subset
\BW^c(\bfB_d)^\star$,~cf~\ref{cFittingVectoriel}.  En revanche, on ne
peut pas vraiment parler d'invariant de MacRae; le module $\bfB_d$
n'est peut-être pas de rang $c$ : il n'y a aucune raison pour que
$\calF_c(\bfB_d)$ soit un idéal fidèle sans hypothèse supplémentaire
sur~$\uP$.
Cependant, nous pouvons définir les invariants $\omegares$
et~$\calR_d$ pour un système quelconque $\uP$ comme étant les
spécialisés de ces mêmes invariants pour la suite générique.

\begin{defn}
Soit $\uP$ quelconque.
Les scalaires $\calR_d(\uP)$ pour $d \geqslant \delta+1$ et la forme linéaire $\omega_{\res,\uP}$ 
sont définis comme étant la spécialisation de ces mêmes objets pour la suite générique.
\end{defn}

\begin{rmq}
\label{Rem-omegares-regulier}
Lorsque la suite $\uP$ est régulière, le $\bfA$-module $\bfB_d$ 
est résoluble de rang $\chi_d$, a fortiori de MacRae.
On peut donc définir \textit{directement} 
ses invariants de MacRae et ceci sans passer par la spécialisation du cas générique.
Le lecteur est alors en droit de se demander si les spécialisés des générateurs des invariants de MacRae
du cas générique sont des générateurs des invariants de MacRae de la définition qualifiée de directe.
La réponse est évidemment oui, et ceci grâce à la propriété de spécialisation~\ref{ExtensionScalaires}.
Une conséquence est la suivante : pour une suite $\uP$ \textit{régulière}, 
\begin{itemize}
\item 
le scalaire $\calR_d$ est \textit{régulier}, 
en tant qu'invariant de MacRae 
du module~$\bfB_d$ ;

\item 
la forme linéaire $\omegares$ est \textit{sans torsion}, 
en tant qu'invariant de MacRae du module $\bfB_\delta$ ;

\item 
le scalaire $\omegares(\nabla)$ est \textit{régulier}, 
en tant qu'invariant de MacRae 
du module~$\bfB'_\delta$.
\end{itemize}
\end{rmq}

\medskip

Comme nous l'avons signalé avant l'énoncé~\ref{PoidsNormalisationMacRae},
nous utiliserons le plus souvent la forme linéaire $\omegares
: \bfA[\uX]_\delta \to \bfA$, remontée de $\uomegares
: \bfB_\delta \to \bfA$ (le fait que $\uomegares$ soit définie
sur le quotient $\bfB_\delta = \bfA[\uX]_\delta/ \uPdelta$ correspondant 
à la nullité de $\omegares$ sur $\uPdelta$).

\medskip
On dispose d'une égalité de sous-modules de $\bfA[\uX]_\delta^\star$
$$
\DVect_{s_\delta}(\Syl_\delta) = \omegares\,\fb_\delta
\qquad \text{où $\fb_\delta$ est un idéal de type fini bien précis}
$$
En évaluant en les $X^\alpha$ cette
égalité, on obtient une factorisation de l'idéal des mineurs d'ordre
$s_\delta$ de $\Syl_\delta$ 
$$
\calD_{s_\delta}(\Syl_\delta) =
\langle\omegares(X^\alpha),\, |\alpha| = \delta\rangle\,\fb_\delta
\ = \ \Im\,\omegares\,\fb_\delta
$$
qui donne en particulier l'inclusion:
$$
\calD_{s_\delta}(\Syl_\delta)  \ \subset \ \Im\,\omegares
$$
Lorsque~$\uP$ est régulière, cette forme linéaire $\omegares$ est (par définition de
$\uomegares$) un pgcd fort des formes linéaires de $\DVect_{s_\delta}(\Syl_\delta)$,
de sorte que, dans l'égalité $\DVect_{s_\delta}(\Syl_\delta) = \omegares\,\fb_\delta$,
l'idéal $\fb_\delta$ est de profondeur $\geqslant 2$.

\subsubsection{Spécialisation de la suite générique à une suite quelconque et idéaux cofacteurs: résumé}
\label{ResumeSpecialisationCasGenerique}

Par spécialisation du  cas générique, on a défini pour une suite $\uP$ les objets suivants:
\begin{quote}
\noindent  
\it 
\begin{itemize}
\item 
la forme linéaire $\omega_{\res,\uP} \in \bfA[\uX]_\delta^\star$ et l'idéal $\fb_\delta = \fb_\delta(\uP)$ de $\bfA$
\begin{enumerate}[$\rhd$]
\item 
$\omega_{\res, \uP}$ s'annule sur $\uPdelta$ et vérifie
$$
\DVect_{s_\delta}(\Syl_\delta)
\ = \ 
\omega_{\res, \uP} \, \fb_\delta
\qquad \text{ ou encore } \qquad 
\FittVect_1(\bfB_\delta) 
\ = \ 
\overline{\omega}_{\res, \uP}
\, \fb_\delta
$$

\item 
Lorsque $\uP$ est régulière, 
la forme linéaire $\omega_{\res,\uP}$ est sans torsion
et l'idéal $\fb_\delta$ est de profondeur~$\geqslant 2$.

\item 
Pour le jeu étalon généralisé, 
$\omega_{\res, \pXD} = p_1^{\widehat d_1-1} \cdots p_n^{\widehat d_n-1} (X^\emouton)^\star$. 
Et le jeu étalon, $\omega_{\res, \uX^D} = (X^\emouton)^\star$. 

\item 
Pour $\uP$ couvrant le jeu étalon généralisé, 
la composante homogène dominante de $\omega_{\res, \uP}(X^\emouton)$ est 
$\omega_{\res, \pXD}(X^\emouton) =
p_1^{\widehat d_1-1} \cdots p_n^{\widehat d_n-1}$.
\end{enumerate}

\item 
pour $d \geqslant \delta+1$, le scalaire $\calR_d(\uP) \in \bfA$  et l'idéal $\fb_d = \fb_d(\uP)$ de $\bfA$

\index{idéal!cofacteur $\fb_d$}%

\begin{enumerate}[$\rhd$]
\item On a les égalités :
$$
\calD_{s_d}(\Syl_d)
\ = \ 
\calR_d(\uP) \, \fb_d
\qquad \text{ ou encore } \qquad 
\calF_0(\bfB_d) 
\ = \ 
\calR_d(\uP) \, \fb_d
$$

\item 
Lorsque $\uP$ est régulière, 
$\calR_d(\uP)$ est un élément régulier de~$\bfA$ 
et l'idéal $\fb_d$ est de profondeur~$\geqslant 2$.

\item 
Pour le jeu étalon généralisé, 
$\calR_d(\pXD) = p_1^{\widehat d_1} \cdots p_n^{\widehat d_n}$.
Et le jeu étalon, $\calR_d(\uX^D) = 1$. 

\item 
Pour $\uP$ couvrant le jeu étalon généralisé, 
la composante homogène dominante de $\calR_d(\uP)$ est 
$\calR_d(\pXD) = p_1^{\widehat d_1} \cdots p_n^{\widehat d_n}$.


\end{enumerate}
\end{itemize}
\end{quote}

\label{NOTA10-fbd}%
%
%

\subsection{Les générateurs de MacRae pour des formats $D$ exceptionnels}

\begin{prop}[La forme $\omegares$ dans les cas d'école] \label{omegaresCasEcole}
Soit $\uP$ un système quelconque.

\begin{enumerate}[\rm i)]
\item Pour $D = (1,\dots, 1,e)$, la forme linéaire $\omegares$ est la forme linéaire $\evalxi$.

\item Pour $n = 2$, la forme linéaire $\omegares$ est la forme linéaire $\omega$.
\end{enumerate}
\end{prop}

\index{théorème!profondeur $\ge 2$ de la forme $\omegaRes{\uP}$ pour des systèmes $\uP$ spécifiques}%
\begin{proof}
Par définition, il suffit de le montrer en générique. Nous traitons les deux points simultanément.

D'après~\ref{FormeLineaireGr2}, si on dispose d'une forme linéaire $\mu : \bfA[\uX]_\delta \to \bfA$
vérifiant 
\begin{center}
$\mu$ nulle sur $\uPdelta$
\qquad \text{ et } \qquad 
$\Gr(\mu) \geqslant 2$
\end{center}
alors cette forme linéaire engendre le sous-module des formes linéaires nulles sur $\uPdelta$.

\noindent
En générique, la forme linéaire $\evalxi$ (resp. $\omega$) 
est une telle forme linéaire (cf.~\ref{evalxiProperties}, resp.~\ref{omegaProperties-n=2}).
Donc~$\omegares$, qui est nulle sur $\uPdelta$, est multiple de $\evalxi$ (resp. $\omega$).
Or les formes linéaires $\omegares$ et 
$\evalxi$ (resp. $\omega$) ont 
même poids en les coefficients de $P_i$ et sont normalisées. 
Donc elles sont égales.
\end{proof}

\medskip

Nous nous intéressons maintenant aux formats $D$ tels que le complexe
$\rmK_{\sbullet,d}(X^D)$, pour $d = \delta$ ou $d = \delta+1$, soit
réduit à $0 \to \rmK_{1,d} \to \rmK_{0,d}$.  Puisque $\rmK_{k,d} =0
\Rightarrow \rmK_{k+1,d} = 0$ (cf. \ref{dminKk}), il faut et il suffit
que $\rmK_{2,d} =0$. Par ailleurs, on a vu en \ref{NulliteJkdKkd}
l'équivalence $\rmK_{h,d} =0\iff\Jex_{h,d} = 0$, si bien que l'on peut
appliquer la proposition~\ref{NulliteJhd} (Nullité de $\Jex_{h,d}$). Ce
qui nous conduit à la condition :
$$
d < \min_{i \ne j} (d_i + d_j)
$$
Enfin, pour $d=\delta$ et $d=\delta+1$, les formats $D$ pour lesquels $\Jex_{2,d}(D) = 0$ ont été
répertoriés en~\ref{J2deltaNul}. Par ailleurs,
la caractéristique d'Euler-Poincaré $\chi_d$ de $\rmK_{\sbullet,d}$
vérifie toujours $\chi_\delta = 1$ et $\chi_{\delta+1} = 0$ et ici
elle est réduite à $\dim \rmK_{0,d} - \dim \rmK_{1,d}$.

\begin{prop}[Pour des formats $D$ particuliers] \label{CasParticuliers}

Soit $\uP$ quelconque.

\begin{enumerate}[\rm i)]
\item 
Lorsque $n=2$ ou ($n=3$ et $d_1, d_2, d_3 \leqslant 2$) 
ou ($n\geqslant 4$ et tous les $d_i = 1$ sauf éventuellement un $d_j = 2$), 
alors $\omegares = \omega$.

\item 
Lorsque $n=2$ ou $D = (1, \dots, 1)$, alors 
$\calR_{\delta+1} = \det W_{1,\delta+1}$.
\end{enumerate}
\end{prop}

\begin{proof} \leavevmode

i) 
Dans tous ces cas, l'application de Sylvester $\Syl_\delta$ 
possède une colonne de moins que son nombre de lignes.
Le module $\DVect_{s_\delta}(\Syl_{\delta})$ est alors \textit{monogène}, engendré 
par la forme linéaire (normalisée) $\omega$, donc $\omegares = \omega$. 

ii) 
Dans les deux cas, l'application de Sylvester $\Syl_{\delta + 1}$ possède autant de colonnes que de lignes.
L'idéal $\calD_{s_{\delta + 1}}(\Syl_{\delta + 1})$ est alors \textit{monogène}, engendré par le seul mineur (normalisé) $\det W_{1,{\delta + 1}}$, 
donc $\calR_{\delta + 1} = \det W_{1,{\delta + 1}}$.
\end {proof}

\subsubsection {Le cas particulier de $D = (2,2,2)$}


Ce format est historiquement célèbre car Cayley (cf. \cite{Cayley1}) et
Sylvester ont étudié le sujet.
Anticipons sur les \emph{définitions} possibles du résultant,
et notamment sur celle en degré $\delta$, à savoir $\Res(\uP) = \omegaRes{\uP}(\nabla_\uP)$
(cf.~\ref{DefResultant}). Plusieurs \emph{expressions} très différentes existent
pour le résultant $\calR$ des 3 polynômes:
$$
\setlength{\tabcolsep}{2pt}
\left\{
\begin{tabular}{rcp{15cm}} 
$P_{1}$ & $=$ & $a_{1}X^{2} + a_{2}XY + a_{3}XZ + a_{4}Y^{2} + a_{5}YZ + a_{6}Z^{2}$\\ [0.1cm] 
$P_{2}$ & $=$ & $b_{1}X^{2} + b_{2}XY + b_{3}XZ + b_{4}Y^{2} + b_{5}YZ + b_{6}Z^{2}$\\ [0.1cm] 
$P_{3}$ & $=$ & $c_{1}X^{2} + c_{2}XY + c_{3}XZ + c_{4}Y^{2} + c_{5}YZ + c_{6}Z^{2}$\\ [0.1cm] 
\end{tabular}
\right.
$$
Pour ce format, on a $\omegares = \omega$, donc $\calR = \omega(\nabla)$.
Voici la forme linéaire $\omega$ :
$$
\omega(\sbullet) \ =\ 
\EastBorderdet{
a_{1}& .& .& b_{1}& \sbullet & c_{1}& .& .& .& . & \Heti{X^{3}} \\ 
a_{2}& a_{1}& .& b_{2}& \sbullet & c_{2}& b_{1}& .& c_{1}& . & \Heti{X^{2}Y} \\ 
a_{3}& .& a_{1}& b_{3}& \sbullet & c_{3}& .& b_{1}& .& c_{1} & \Heti{X^{2}Z} \\ 
a_{4}& a_{2}& .& b_{4}& \sbullet & c_{4}& b_{2}& .& c_{2}& . & \Heti{XY^{2}} \\ 
a_{5}& a_{3}& a_{2}& b_{5}& \sbullet & c_{5}& b_{3}& b_{2}& c_{3}& c_{2} & \Heti{XYZ} \\ 
a_{6}& .& a_{3}& b_{6}& \sbullet & c_{6}& .& b_{3}& .& c_{3} & \Heti{XZ^{2}} \\ 
.& a_{4}& .& .& \sbullet & .& b_{4}& .& c_{4}& . & \Heti{Y^{3}} \\ 
.& a_{5}& a_{4}& .& \sbullet & .& b_{5}& b_{4}& c_{5}& c_{4} & \Heti{Y^{2}Z} \\ 
.& a_{6}& a_{5}& .& \sbullet & .& b_{6}& b_{5}& c_{6}& c_{5} & \Heti{YZ^{2}} \\ 
.& .& a_{6}& .& \sbullet & .& .& b_{6}& .& c_{6} & \Heti{Z^{3}} \\ 
}
$$
et le déterminant bezoutien \og rigidifié\fg{} (cf après la définition \ref{DefNabla})
$$
\nabla =
\begin {vmatrix}
a_1X + a_2Y + a_3Z &b_1X + b_2Y + b_3Z &c_1X + c_2Y + c_3Z \\
a_4Y + a_5Z        &b_4Y + b_5Z        &c_4Y + c_5Z        \\
a_6Z               &b_6Z               &c_6Z  \\
\end {vmatrix}
$$
Une manière de calculer $\nabla$ consiste à développer le produit extérieur des colonnes:
$$
\big((a_1X + a_2Y + a_3Z)e_1 + (a_4Y + a_5Z)e_2 + a_6Ze_3\big) \wedge
\big(\text{colonne en $b_i$}\big) \wedge \big(\text{colonne en $c_i$}\big) 
$$
Le développement donne $3 \times 2 \times 1 = 6$ termes. Et en regardant
par exemple la colonne 1 dans les yeux:
$$
\nabla = [146]XYZ + [246]Y^2Z + [346]YZ^2 + [156]XZ^2 + [256]YZ^2 + [356]Z^3
\quad \text{où}\quad
[ijk] = \begin{vmatrix} a_i&b_i&c_i\\ a_j&b_j&c_j\\ a_k&b_k&c_k\\ \end{vmatrix}
$$
L'expression $\calR = \omega(\nabla)$ 
est donc donnée par le déterminant $10 \times 10$ :
$$
\newcommand \petit[1] {\mbox{\hspace{-0.3cm}\footnotesize$#1$\hspace{-0.3cm}}}
\calR \ =\ 
\EastBorderdet{
a_{1}& .& .& b_{1}& . & c_{1}& .& .& .& . & \Heti{X^{3}} \\ 
a_{2}& a_{1}& .& b_{2}& . & c_{2}& b_{1}& .& c_{1}& . & \Heti{X^{2}Y} \\ 
a_{3}& .& a_{1}& b_{3}& . & c_{3}& .& b_{1}& .& c_{1} & \Heti{X^{2}Z} \\ 
a_{4}& a_{2}& .& b_{4}& . & c_{4}& b_{2}& .& c_{2}& . & \Heti{XY^{2}} \\ 
a_{5}& a_{3}& a_{2}& b_{5}& \petit{[146]} & c_{5}& b_{3}& b_{2}& c_{3}& c_{2} & \Heti{XYZ} \\ 
a_{6}& .& a_{3}& b_{6}& \petit{[156]} & c_{6}& .& b_{3}& .& c_{3} & \Heti{XZ^{2}} \\ 
.& a_{4}& .& .& . & .& b_{4}& .& c_{4}& . & \Heti{Y^{3}} \\ 
.& a_{5}& a_{4}& .& \petit{[246]} & .& b_{5}& b_{4}& c_{5}& c_{4} & \Heti{Y^{2}Z} \\ 
.& a_{6}& a_{5}& .& \petit{[346]\!+\![256]} & .& b_{6}& b_{5}& c_{6}& c_{5} & \Heti{YZ^{2}} \\ 
.& .& a_{6}& .& \petit{[356]} & .& .& b_{6}& .& c_{6} & \Heti{Z^{3}} \\ 
}
$$

\medskip

Nous verrons également (cf. section~\ref{sousSectionMacaulayRecurrence})
qu'il existe une formule du résultant $\calR$ 
en degré $\delta+1 = 4$ comme quotient d'un déterminant $15 \times 15$ par
un déterminant $3 \times 3$, à savoir :
$$
\calR = \frac{\det W_{1,\delta+1}} {\det W_{2,\delta+1}}  
\qquad\qquad
\text{où } 
\qquad 
W_{2,\delta+1} = 
\EastBordermatrix{
  a_1 & . & . & \Heti{X^2Y^2} \\
  . & a_1 & b_1 & \Heti{X^2Z^2} \\
  a_6 & a_4 & b_4 & \Heti{Y^2Z^2} \\
}
$$
et
$$
W_{1,\delta+1} \ = \ 
\EastBordermatrix{
a_{1} & . & . & . & . & . & . & . & . & . & . & . & . & . & . & \Heti{X^{4}} \\ 
a_{2} & a_{1} & . & . & . & . & b_{1} & . & c_{1} & . & . & . & . & . & . & \Heti{X^{3}Y} \\ 
a_{3} & . & a_{1} & . & . & . & . & b_{1} & . & c_{1} & . & . & . & . & . & \Heti{X^{3}Z} \\ 
a_{4} & a_{2} & . & a_{1} & . & . & b_{2} & . & c_{2} & . & b_{1} & . & . & . & . & \Heti{X^{2}Y^{2}} \\ 
a_{5} & a_{3} & a_{2} & . & a_{1} & . & b_{3} & b_{2} & c_{3} & c_{2} & . & b_{1} & . & c_{1} & . & \Heti{X^{2}YZ} \\ 
a_{6} & . & a_{3} & . & . & a_{1} & . & b_{3} & . & c_{3} & . & . & b_{1} & . & c_{1} & \Heti{X^{2}Z^{2}} \\ 
. & a_{4} & . & a_{2} & . & . & b_{4} & . & c_{4} & . & b_{2} & . & . & . & . & \Heti{XY^{3}} \\ 
. & a_{5} & a_{4} & a_{3} & a_{2} & . & b_{5} & b_{4} & c_{5} & c_{4} & b_{3} & b_{2} & . & c_{2} & . & \Heti{XY^{2}Z} \\ 
. & a_{6} & a_{5} & . & a_{3} & a_{2} & b_{6} & b_{5} & c_{6} & c_{5} & . & b_{3} & b_{2} & c_{3} & c_{2} & \Heti{XYZ^{2}} \\ 
. & . & a_{6} & . & . & a_{3} & . & b_{6} & . & c_{6} & . & . & b_{3} & . & c_{3} & \Heti{XZ^{3}} \\ 
. & . & . & a_{4} & . & . & . & . & . & . & b_{4} & . & . & . & . & \Heti{Y^{4}} \\ 
. & . & . & a_{5} & a_{4} & . & . & . & . & . & b_{5} & b_{4} & . & c_{4} & . & \Heti{Y^{3}Z} \\ 
. & . & . & a_{6} & a_{5} & a_{4} & . & . & . & . & b_{6} & b_{5} & b_{4} & c_{5} & c_{4} & \Heti{Y^{2}Z^{2}} \\ 
. & . & . & . & a_{6} & a_{5} & . & . & . & . & . & b_{6} & b_{5} & c_{6} & c_{5} & \Heti{YZ^{3}} \\ 
. & . & . & . & . & a_{6} & . & . & . & . & . & . & b_{6} & . & c_{6} & \Heti{Z^{4}} \\ 
}
$$

\bigskip

Terminons par une formule attribuée à Sylvester à l'aide d'un déterminant $6\times 6$.
Voir par exemple~\cite{Cox}, formule (2.8) de la section $\S2$ du chapitre~3
et l'exercice~22 de la section~$\S4$ ;
ou bien \cite{Sturmfels2} (attention, oubli d'une puissance
de $2$ dans la formule) ou encore \cite{AD}, section 2.7.

Il y intervient le jacobien $J$ des trois polynômes $P_1, P_2, P_3$ qui est un polynôme homogène en
$X,Y,Z$ de degré 3. On note $(x_1, \dots,
x_6)$ les coordonnées de $\frac{\partial J}{\partial X}$ dans la
base monomiale $(X^2,XY,XZ,Y^2,YZ,Z^2)$  de $\bfA[X,Y,Z]_2$ 
utilisée pour les polynômes $P_i$.
Idem pour les $y_i$ et $z_i$ en remplaçant
$\frac{\partial}{\partial X}$ par $\frac{\partial}{\partial Y}$ et
$\frac{\partial}{\partial Z}$. 
L'expression du résultant est alors :
$$
\calR 
\ = \ 
\frac{1}{2^9} \, 
\NorthEastBorderdet{
\Veti{P_1} & \Veti{\partial J/\partial Z} & \Veti{\partial J/\partial Y} & 
\Veti{P_2} & \Veti{\partial J/\partial X} & \Veti{P_3} & \\
  a_1 & z_1 & y_1 & b_1 & x_1 & c_1 & \Heti{X^2} \\
  a_2 & z_2 & y_2 & b_2 & x_2 & c_2 & \Heti{XY} \\
  a_3 & z_3 & y_3 & b_3 & x_3 & c_3 & \Heti{XZ} \\
  a_4 & z_4 & y_4 & b_4 & x_4 & c_4 & \Heti{Y^2} \\
  a_5 & z_5 & y_5 & b_5 & x_5 & c_5 & \Heti{YZ} \\
  a_6 & z_6 & y_6 & b_6 & x_6 & c_6 & \Heti{Z^2} \\  
}
\ = \ 
\frac{1}{2^9} \, 
\NorthEastBorderdet{
\Veti{P_1} & \Veti{P_2} & \Veti{P_3} & 
\Veti{\partial J/\partial X} & \Veti{\partial J/\partial Y}  & 
\Veti{\partial J/\partial Z} & \\
 a_1 & b_1 & c_1 & x_1 & y_1 & z_1 & \Heti{X^2} \\
 a_4 & b_4 & c_4 & x_4 & y_4 & z_4 & \Heti{Y^2} \\
 a_6 & b_6 & c_6 & x_6 & y_6 & z_6 & \Heti{Z^2} \\
 a_5 & b_5 & c_5 & x_5 & y_5 & z_5 & \Heti{YZ} \\
 a_3 & b_3 & c_3 & x_3 & y_3 & z_3 & \Heti{XZ} \\
 a_2 & b_2 & c_2 & x_2 & y_2 & z_2 & \Heti{XY} \\  
}
$$
Nous ne prétendons pas avoir prouvé cette formule \emph {pour l'instant}.
Nous avons fourni deux visages de la même matrice (nous aurions pu en fournir
d'autres) pour insister sur son aspect strictement carré,
défini par la correspondance 
$\begin {smallmatrix}
P_1 &P_2 &P_3 &\partial/\partial X&\partial/\partial Y&\partial/\partial Z\\
X^2 &Y^2 &Z^2 &YZ                           &XZ & XY \\
\end {smallmatrix}
$.
Une petite vérification avec le jeu étalon généralisé $(p_1X^2, p_2Y^2, p_3Z^2)$.
Sa matrice jacobienne est 
$\left[\begin{smallmatrix}
2p_1X &  0  &   0 \\
 0  & 2p_2Y  & 0 \\
0  & 0 & 2p_3Z\\
\end{smallmatrix}\right]
$, 
de déterminant $2^3 p_1p_2p_3XYZ$.
Le résultant vaut $p_1^4p_2^4p_3^4$ (confer le chapitre~\ref{ChapJeuEtalonGeneralise} consacré au jeu étalon généralisé) 
et la formule de Sylvester redonne bien cette valeur:
$$
\frac{1}{2^9} \, 
\NorthEastBorderdet{
\Veti{P_1} & \Veti{\partial J/\partial Z} & \Veti{\partial J/\partial Y} & 
\Veti{P_2} & \Veti{\partial J/\partial X} & \Veti{P_3} & \\
  p_1 & . & . & . & . & . & \Heti{X^2} \\
  . & 2^3p_1p_2p_3 & . & . & . & . & \Heti{XY} \\
  . & . & 2^3p_1p_2p_3 & . & . & . & \Heti{XZ} \\
  . & . & . & p_2 & . & . & \Heti{Y^2} \\
  . & . & . & . & 2^3p_1p_2p_3 & . & \Heti{YZ} \\
  . & . & . & . & . & p_3 & \Heti{Z^2} \\  
}
\ = \ 
p_1^4p_2^4p_3^4
$$

\subsection{Des propriétés remarquables de la forme $\omegaRes{\protect\uP}:\bfA[\protect\uX]_\delta \to\bfA$}
\label{SectionomegaresProperties}
\index{forme linéaire!intrinsèque $\omegares=\omegaRes{\uP}$}%

\subsubsection*{$\bullet$ Propriété Cramer}

\begin{theo} \label{omegaresCramer}
Soit $\uP$ un système quelconque.
La forme linéaire $\omegares$ attachée à $\uP$ est de Cramer :
$$
\forall\,  F, G \in \bfA[\uX]_\delta, \qquad 
\omegares(G)F - \omegares(F)G \ \in \ 
\langle\uP\rangle_\delta
$$
\end{theo}

\index{propriété Cramer!2@d'une forme linéaire}%

\begin{proof}
Il suffit de démontrer le résultat pour la suite générique.
La forme linéaire $\omega$ est multiple de la forme linéaire $\omegares$: on a $\omega = b_\delta\,\omegares$.
Comme $\omega$ est de Cramer (cf.~\ref{omegaProprietes}), il suffit de montrer que $b_\delta$ est 
$\bfB_\delta$-régulier. 
En évaluant en $X^\emouton$, on obtient $\det W_{1,\delta} = b_\delta\, \omegares(X^\emouton)$.
Ainsi $b_\delta$ est un diviseur de $\det W_{1,\delta}$ qui est 
$\bfB_\delta$-régulier (cf.~\ref{GenPsatRegScalars}), 
donc est lui-même $\bfB_\delta$-régulier.
\end{proof}

\subsubsection*{$\bullet$ Injectivité de la forme $\uomegares : \bfB_\delta \to \bfA$}
\begin{theo}
\label{omegaresInjectiveDoncCramer}

Soit $\uP$ la suite {\rm \bf générique}.

La forme linéaire $\uomegares$ est injective, 
c'est-à-dire $\Ker \omegares = \uPdelta$.
\end{theo}

\begin{proof}
Soit $F \in \bfA[\uX]_\delta$ tel que $\omegares(F) = 0$. Puisque $\omega$ est multiple de
$\omegares$, on a $\omega(F) = 0$. Or le noyau de $\omega$ est $\uPdelta$ en générique
d'après~\ref{omegaInjective} : on obtient donc $F \in \uPdelta$.
\end{proof}

\begin{rmq}
En terrain générique, la propriété Cramer se déduit de l'injectivité de $\uomegares$
puisque $\omegares(F) G - \omegares(G) F \in \Ker \omegares$. Et comme la propriété
Cramer se spécialise, on l'obtient pour tout~$\uP$.
Bien entendu, la propriété d'injectivité de $\uomegares$ ne se spécialise pas, confer
par exemple le jeu étalon généralisé.
\end{rmq}

\subsubsection*{$\bullet$ Profondeur de l'idéal $\Im\omegares$}

Pour une suite régulière $\uP$, la propriété d'injectivité de
$\uomegares$ lorsqu'elle est vérifiée, fournit (cf. le théorème qui
vient) une inégalité de profondeur qui va s'avérer capitale:
$\Gr(\Im\omegares) \geqslant 2$. Mais ces deux propriétés de
$\uomegares$ (injectivité, profondeur de l'image) sont bien distinctes
comme en témoigne l'exemple de la section
suivante~\ref{ExempleProfondeurNonInjective}.

\index{théorème!profondeur $\ge 2$ de la forme $\omegaRes{\uP}$ pour des systèmes $\uP$ spécifiques}%

\begin{theo} \label{omegaresGr2}

Soit $\uP$ une suite \emph {régulière}. 

\begin{enumerate}[\rm i)]
\item
On suppose $\uomegares$ injective (c'est le cas par exemple de la suite générique).
  
Alors l'image de la forme ${\omegares}$ et l'idéal
déterminantiel d'ordre $s_\delta = \dim\bfA[\uX]_\delta-1$ de
$\Syl_\delta$, liés par l'inclusion $\calD_{s_\delta}(\Syl_\delta)
\subset \Im\omegares$, sont des idéaux de profondeur $\geqslant 2$:
$$
\Gr(\omegares) \geqslant 2, \qquad
\Gr\big(\calD_{s_\delta}(\Syl_\delta)\big) \geqslant 2
$$
  
\item
On suppose $\Gr(\omegares) \geqslant 2$.
  
Alors le module $\bfB_\delta^\star$ est libre de rang $1$ engendré par la forme linéaire $\uomegares$.

Remonté au niveau $\bfA[\uX]_\delta$, cela s'énonce : 
la forme linéaire $\omegares$ est une base de $\Ker \transpose {\Syl_\delta}$.
\end{enumerate}
\end{theo}

\begin{proof} \leavevmode

En ce qui concerne les deux idéaux intervenant dans le premier point,
voici quelques rappels sous forme d'une variation des résultats
fournis en page~\pageref{ResumeSpecialisationCasGenerique}.  Tout
d'abord, sans aucune hypothèse sur $\uP$, on dispose d'une
factorisation du type $\BW^{s_\delta}(\Syl_\delta) = \Theta_\res\,\nu$
avec $\nu : \BW^{s_\delta}(\rmK_{1,\delta}) \to \bfA$ et $\Theta_\res
\in \BW^{s_\delta}(\rmK_{0,\delta})$.  Comme $\Theta_\res$ et
$\omegares$ se correspondent par un isomorphisme déterminantal, ils
ont même idéal contenu.  Lorsque $\uP$ est régulière, la factorisation
vérifie $\Gr(\nu) \geqslant 2$ et $\Gr(\Theta_\res) \geqslant 1$; en
prenant les idéaux contenus des vecteurs intervenant dans la
factorisation, on obtient donc:
$$
\calD_{s_\delta}(\Syl_\delta) = \fb_\delta\, \Im\omegares
\qquad \text{avec}\qquad \fb_\delta = \rmc(\nu)
$$
Puisque $\Gr(\fb_\delta) \geqslant 2$, on voit donc que les deux inégalités de profondeur
annoncées dans le point i) sont équivalentes, en vertu de~\ref{Gr2ProprietesElementaires} et~\ref{Gr2Multiplicativity}.

\smallskip
Note: on montrera dans le théorème \ref{omegaresVersusSyldelta}, sans aucune hypothèse
sur $\uP$, que les idéaux $\calD_{s_\delta}(\Syl_\delta)$ et $\Im\omegares$ ont même racine
et de ce fait (cf. le corollaire~\ref{GrCompatibleInclusionRacine}) même profondeur.

\medskip
i)
L'argument principal consiste à augmenter le complexe $\rmK_{\sbullet,\delta}(\uP)$
à l'aide de $\omegares : \rmK_{0,\delta} \to \bfA$.  
Les détails sont fournis dans le lemme général d'algèbre commutative ci-après.
Certaines redites sont voulues pour améliorer la lisibilité
d'une formulation plus abstraite.

\medskip

ii) La suite $\uP$ étant supposée régulière, le $\bfA$-module
$\bfB_\delta$ est de rang 1 et le résultat peut être obtenu comme
conséquence de~\ref{FormesLineairesRang01}-ii).  On propose une
variante basée sur la propriété Cramer de $\uomegares$ et qui a
l'avantage d'utiliser uniquement $\Gr(\uomegares) \geqslant 2$ sans
supposer~$\uP$ régulière. En notant $x^\alpha$ l'image
dans~$\bfB_\delta$ d'un monôme $X^\alpha$ de degré $\delta$:
$$
\forall\,  |\alpha| = |\beta| = \delta, \qquad 
\uomegares(x^\alpha)\,x^\beta  = \uomegares(x^\beta)\,x^\alpha \qquad\qquad
$$
Soit $\mu : \bfB_\delta \to \bfA$ une forme linéaire. En appliquant $\mu$ à
cette égalité Cramer, il vient
$$
\uomegares(x^\alpha)\,\mu(x^\beta)  = \uomegares(x^\beta)\,\mu(x^\alpha)
$$
Par définition même de la profondeur $\geqslant 2$, il existe $q \in \bfA$ tel
que $\mu(x^\alpha) = q\,\uomegares(x^\alpha)$ donc $\mu = q\,\uomegares$.
\end{proof}

En ce qui concerne le premier point du théorème précédent, nous en 
fournissons deux approches dans des contextes légèrement différents.
Dans la première, le module $M$ est supposé librement résoluble de
rang~1 (c'est le cas de $M = \bfB_\delta$ lorsque $\uP$ est régulière)
et \og la \fg{} forme linéaire dont il est question est celle 
qui intervient dans~\ref{LibrementResolubleImpliqueMacRae} ;
tandis que dans la seconde, plus générale, $M$ est seulement supposé
de MacRae de rang~1 et la forme linéaire qui intervient est définie en~\ref{DefModuleMacRae}.

\begin{lem}[Injectivité versus profondeur]
\label{InjectivityToGr2}
Soit $M$ un module librement résoluble de rang $1$, de forme linéaire $\vartheta$.
Si $\vartheta$ est injective, alors $\Gr(\vartheta) \geqslant 2$. 
\end{lem}

\begin{proof}
Soit $(F_\sbullet, u_\sbullet)$ une résolution libre de $M$ de sorte que la différentielle $u_1$ 
est de rang $r_1 = \dim F_0 -1$. On augmente cette résolution en :
$$
\xymatrix{
F_1 \ar[r]^-{u_1} & F_0 \ar[r]^-{\mu} & \bfA
}
$$
où $\mu$ est la forme linéaire composée $F_0 \twoheadrightarrow M
\xrightarrow{\vartheta} \bfA$ de sorte que $\Im \mu = \Im \vartheta$
(autrement dit, $\mu$ est un relèvement de $\vartheta$ à $F_0$ et
$\vartheta = \overline \mu$). Il s'agit donc de montrer que $\Gr(\mu)
\geqslant 2$.  L'injectivité de $\vartheta$ se traduit par $\Im u_1 =
\Ker \mu$.  L'exactitude du complexe augmenté fournit d'après \ref
{WhatMakesAComplexExact} (What makes a complex exact?)
l'inégalité $\Gr\big(\calD_{r_1}(u_1)\big) \geqslant 2$.  Or on
dispose de l'inclusion $\calD_{r_1}(u_1) \subset \Im \mu$ (cf le
rappel ci-après) d'où l'on déduit $\Gr(\mu) \geqslant 2$.

Rappelons pourquoi $\calD_{r_1}(u_1) \subset \Im \mu$ 
(cf. aussi~page~\pageref{DVectVERSUScalD}).
Par définition de $\mu$, on dispose d'une égalité $\DVect_{r_1}(u_1) = \mu\,\fb$ de sous-modules de~$F_0^\star$,
qui, par évaluation sur $F_0$, fournit l'égalité d'idéaux de $\bfA$ :
$$
\calD_{r_1}(u_1) = (\Im\mu)\,\fb
$$
a fortiori l'inclusion annoncée.
\end{proof}

Ci-dessous, la machinerie des résolutions libres est remplacée par celle de la profondeur~2.

\begin{lem}[Injectivité versus profondeur, bis] \label{InjectivityToGr2bis}

Soit $M$ un module de MacRae de rang $1$ et $\vartheta \in M^\star$ sa forme linéaire invariant
de MacRae que l'on suppose \emph {injective}.

\begin{enumerate}[\rm i)]
\item 
L'idéal image de $\vartheta$ est de profondeur $\geqslant 2$.

\item 
La forme linéaire $\vartheta$ possède la propriété de Cramer: $\vartheta(m)m' = \vartheta(m')m$
pour tous $m,m' \in M$.
\end{enumerate}
\end{lem}

\index{propriété Cramer!2@d'une forme linéaire}%

\begin {proof} \leavevmode

i) Pour montrer $\Gr(\vartheta) \geqslant 2$, d'après le
principe local-global pour la profondeur 2 (point iv de la
proposition~\ref{PropProfondeur2}) il suffit de le faire après localisation en
des éléments d'une suite $\us$ de profondeur~$\geqslant
2$. Considérons celle qui intervient dans le
théorème \ref{1ElementaryPresentation} (présentation 1-élémentaire locale d'un
module de MacRae de rang 1) et localisons en un scalaire
$s$ de $\us$. Sur le localisé $\bfA_s$, on dispose donc d'une suite
exacte du type:
$$
0 \to \xymatrix {\bfA_s^r\ar[r]^-u  & \bfA_s^{r+1}\ar[r]^\pi & M_s} \to 0
$$
Et la forme linéaire $\vartheta_s : M_s \to \bfA_s$ est par hypothèse
injective (localisation de l'injectivité de $\vartheta$). En désignant
par $\mu : \bfA_s^{r+1} \to \bfA_s$ la forme linéaire des mineurs
signés de $u$ ($\mu$ est un relèvement de $\vartheta_s$ au sens où
$\vartheta_s \circ \pi = \mu$), l'injectivité de $\vartheta_s$ se traduit par $\Im u
= \Ker \mu$ ; d'où l'exactitude en $\bfA_s^{r+1}$ de:
$$
0 \to \xymatrix {\bfA_s^r\ar[r]^-u  & \bfA_s^{r+1}\ar[r]^\mu & \bfA_s} 
$$
Le théorème d'Hilbert-Burch (cf.~\ref{sectionHB}) fournit alors
$\Gr(\mu) \geqslant 2$ \idest{} $\Gr(\vartheta_s) \geqslant 2$.

\medskip
ii) Le vecteur $\vartheta(m)m' - \vartheta(m')m$ est dans le noyau de 
la forme linéaire $\vartheta$, injective par hypothèse. Ainsi, ce vecteur est nul.
\end {proof}

\begin {rmq}
Dans le contexte des deux derniers lemmes, 
il n'y a pas de réciproque au sens où l'on
peut avoir $\Gr(\vartheta) \geqslant 2$ \emph {sans que} $\vartheta$
ne soit injective. Dans le cadre de l'élimination (celui qui nous
intéresse le plus), nous allons en donner un exemple dans la
section suivante.

\smallskip

En dehors de ce cadre, c'est facile d'en donner des exemples.
Considérons le $\bfA$-module $M = \bfA \oplus M_0$ où $M_0 =
\bfA/\langle \ub \rangle $ avec $\ub = (b_1, b_2)$ de profondeur
$\geqslant 2$.  Il est présenté par
$$
\xymatrix @C=4pc{  
u : \quad 
\bfA^2 \ar[r]^-{ 
\left[
\begin{smallmatrix}
0 &  0 \\
b_1 & b_2
\end{smallmatrix}
\right]
}
& \bfA \oplus \bfA^1
}
$$
On a vu (cf.~\ref{AoplusM0}) que ce module est de MacRae de rang $1$ 
et que $\MacRaeVect(M)$ est engendré par la première projection 
$\vartheta : M = \bfA \oplus M_0 \to \bfA$. 
Cette forme linéaire est surjective \idest{} de profondeur infinie (donc supérieure à $2$), et 
pourtant, elle n'est pas injective: son noyau est $0 \oplus M_0$, 
non nul dès lors que $\ub = (b_1, b_2)$ n'est pas unimodulaire.

\smallskip

Une autre \og pathologie\fg{} pour ce même exemple : la forme linéaire $\vartheta$ n'a aucune
raison d'être de Cramer.  Puisque $\vartheta : M = \bfA \oplus M_0 \to
\bfA$ est définie par $\vartheta(a + m_0) = a$, l'égalité de Cramer s'écrit
$\vartheta(a' + m'_0) (a + m_0) = \vartheta(a + m_0) (a' + m'_0)$
et équivaut donc à $a'm_0 = am'_0$.  Cette dernière égalité a lieu pour
tous $a,a' \in \bfA$, $m_0, m'_0 \in M_0$ si et seulement si $M_0 =
0$, ce qui équivaut à $\ub$ unimodulaire. Bilan : 
\begin{center}
$\vartheta$ de Cramer $\iff$ $M_0 = 0$ $\iff$ $\ub$ unimodulaire $\iff$ $\vartheta$ injective
\end{center}
Signalons une généralisation de ce petit contre-exemple.  Soit $M_0$
un module de MacRae de rang $0$ d'invariant de MacRae trivial
$\MacRae(M_0) = \bfA$.  Alors le module $M = \bfA \oplus M_0$ est de
MacRae de rang $1$. En notant $\vartheta \in M^\star$ une forme linéaire
engendrant l'invariant de MacRae $\MacRaeVect(M)$,  on a les équivalences:
\begin{center}
$\vartheta$ de Cramer $\iff M_0 = 0 \iff$ $\vartheta$ injective
\end{center}
\end {rmq}

\begin {rmqs} [Où l'on brode autour de $\Gr(\omegaRes{\uP}) \geqslant 2 \nRightarrow \uP$ régulière]

  Il est légitime de se poser des questions du type: lorsque $\Gr(\omegaRes{\uP}) \geqslant 2$,
peut-on en déduire que $\uP$ est régulière? La réponse est négative.

Nous allons donner comme exemple celui du jeu circulaire, objet du
chapitre~\ref{ChapJeuCirculaire}. Pour cette raison, nous changeons
la notation $\uP$ en $\uQ$ en rappelant que $Q_i = X_i^{d_i-1}(X_i -
X_{i+1})$ avec une notation circulaire des indices.  Nous allons en
profiter pour examiner la négation \og $\uQ$ n'est \emph {pas} régulière\fg{}
et disserter sur le fait d'en tirer des énoncés positifs.
Ceci permet en prime de réaliser quelques retours en arrière.

\medskip  
$\bullet$
On a bien $\Gr(\omegaRes{\uQ}) \geqslant 2$ (et même plus). En effet, dans
la proposition \ref{ProprietesJeuCirculaire}, on a cerné la
forme~$\omega_\uQ$, en particulier on a vu qu'elle est surjective
(elle vaut 1 en chaque monôme). Ceci a comme conséquence d'une part
l'égalité $\omegaRes{\uQ} = \omega_\uQ$ et d'autre part
$\Gr(\omegaRes{\uQ}) = \infty$.  De plus, dans cette même proposition, on a vu
que $\Ker \omega_\uQ = \langle\uQ\rangle_\delta$ donc
${\uomegares}_{,\uQ}$ est injective.

\smallskip
Note pour les amateurs de preuve directe. Soit $\uP$ un système tel que
$\omegaRes{\uP}$ soit surjective. Montrer directement que
${\uomegares}_{,\uP}$ est une base de $\bfB_\delta^\star$ (s'inspirer de
la preuve du point ii du théorème~\ref{omegaresGr2}).

\medskip
$\bullet$
Nous allons donner plusieurs arguments permettant de montrer que la
suite $\uQ$ n'est pas régulière.  Anticipons un peu en supposant le
résultant défini (il le sera définitivement dans le prochain chapitre)
et en supposant acquises deux de ses propriétés.  La première est que le résultant
d'une suite régulière est un scalaire régulier (cf la proposition
\ref{ProprietesBasiquesResultant} qui reprend la
remarque~\ref{Rem-omegares-regulier} en
page~\pageref{ResumeSpecialisationCasGenerique} puisque $\calR_d(\uQ)$
est un visage possible de $\Res(\uQ)$ pour $d \geqslant \delta+1$).  Or ici,
on a $\Res(\uQ) = 0$ car le point $\Un = (1 : \cdots : 1)$ est un zéro
commun du système (propriété classique connue de tous, énoncée en~\ref
{ZeroCommunNulliteRes}). L'argument de nullité du résultant a
l'avantage de s'appliquer à tout système $\uX^D-\uR$ étudié dans la
section \ref{SectionJeuxSimples}: puisque $R_i$ est un monôme de
degré $d_i$, le point $\Un$ est un zéro de ce système, donc
$\Res(\uX^D - \uR) = 0$; a fortiori un tel système n'est jamais une
suite régulière.

\medskip
$\bullet$
On peut se passer du résultant de la manière suivante.  Puisque $\Un$
est un zéro commun de $\uQ$, alors $\langle\uQ\rangle^\sat \cap \bfA =
0$ (cf le commentaire après la définition de l'idéal d'élimination
dans le chapitre~\ref{ObjetsSuiteP}).  En rappelant que tout mineur
plein de $\Syl_{\delta+1}(\uQ)$ est dans l'idéal d'élimination, on en
déduit~${\calF_0(\bfB_{\delta+1}) = 0}$. Ainsi,
$\calF_0(\bfB_{\delta+1})$ étant loin d'être fidèle, le module
$\bfB_{\delta+1}$ n'est pas de rang $0$, ce qui constitue une
obstruction au fait que $\uQ$ est régulière.

Pour les amateurs de preuve directe (bis): on peut obtenir que
$\calF_0(\bfB_{\delta+1}) = 0$ en montrant que la somme des lignes de
$\Syl_{\delta+1}(\uQ)$ est nulle.

Bien entendu, puisque le système $\uQ$ est à coefficients dans
l'anneau principal $\bbZ$, le $\bbZ$-module $\bfB_{\delta+1}$ est
librement résoluble mais pas via la composante homogène de degré
$\delta+1$ du complexe de Koszul de~$\uQ$! (voir ci-après).
En fait, le module $\bfB_{\delta+1}$ est libre de rang 1. Pour le voir,
on laisse le soin au lecteur de démontrer que:
$$
\langle\uQ\rangle_{\delta+1} = \Ker \eval_\Un =
\sum_{|\alpha| = |\beta| = \delta+1} \bfA.(X^\alpha - X^\beta)
$$
Fixons un $X^\beta$ de degré $\delta+1$. Alors:
$$
\langle\uQ\rangle_{\delta+1} = \bigoplus_{|\alpha| = \delta+1} \bfA.(X^\alpha - X^\beta),
\qquad
\bfA[\uX]_{\delta+1} = \langle \uQ\rangle_{\delta + 1} \oplus \bfA\, X^\beta
$$
On en déduit que $\bfB_{\delta+1}$ est libre de rang 1 avec comme base
la classe de n'importe quel monôme de degré~$\delta+1$.

\medskip
$\bullet$
Une autre justification consiste à utiliser le fait que tout déterminant
bezoutien $\nabla$ d'une suite régulière est un polynôme régulier
(cf. \ref{BezoutienRegulier}), ou bien le fait que 
$\overline{\nabla} \in \bfB_\delta$ est sans torsion
(cf.~\ref{WiebeCech}). 

Or ici, nous allons exhiber un déterminant bezoutien $\nabla_\uQ$ qui est nul, obstruction
au fait que $\uQ$ soit régulière ; et qui confirme que $\Res(\uQ) = 0$, car
$\omegaRes{\uQ}(\nabla_\uQ)$ est un autre visage de $\Res(\uQ)$.

Le déterminant bezoutien rigidifié (cf après la définition \ref{DefNabla}) n'est pas
nul et ne convient donc pas. 
Voici la matrice bezoutienne rigidifiée (pour $n=4$) lorsque certains $d_i \geqslant 2$:
$$
\begin {bmatrix}
X_1^{d_1-1} - X_2^{d_1-1} &.                    &.                        &X_4^{d_4-1} \\
.                      &X_2^{d_2-2}(X_2-X_3)  &.                        &.        \\
.                       &.                    &X_3^{d_3-2}(X_3-X_4)  &.         \\
.                       &.                    &.                         &X_4^{d_4-1} \\
\end {bmatrix}
$$
Nous lui préférons la matrice bezoutienne ci-dessous qui présente plusieurs avantages:
$$
\dsV = \begin {bmatrix}
X_1^{d_1-1}  &.            &.           &-X_4^{d_4-1} \\
-X_1^{d_1-1} &X_2^{d_2-1}   &.           &. \\
.           &-X_2^{d_2-1}  &X_3^{d_3-1}  &. \\
.           &.            &-X_3^{d_3-1} &X_4^{d_4-1} \\  
\end {bmatrix}
$$
Puisque la somme des lignes est nulle, on voit aussitôt que $\nabla_\uQ := \det(\dsV) = 0$.

\medskip
$\bullet$
En multipliant à droite $[\uQ] = [\uX].\dsV$ par $\widetilde \dsV$, on obtient $n$
relations homogènes de degré $\delta+1$ données par $[\uQ].\widetilde\dsV = 0$. En fait une seule
puisque, pour une matrice dont la somme des lignes est nulle, les colonnes de sa cotransposée sont
identiques et de composantes les mineurs principaux (considérer la transposée et lui appliquer
la proposition~\ref{riliSumRelation}). Ici:
$$
M = \begin {bmatrix}
a_1  &.     &.    &-a_4 \\
-a_1 &a_2   &.    &. \\
.    &-a_2  &a_3  &. \\
.    &.     &-a_3 &a_4 \\  
\end {bmatrix}
\qquad \text{col. commune de $\widetilde M$}
\def\wha{\widehat a}
= \begin {bmatrix}
\wha_1 \\
\wha_2 \\
\wha_3 \\
\wha_4 \\
\end {bmatrix}
\quad \text{où} \quad
\widehat a_i = a_1 \cdots a_{i-1} a_i\kern-6pt/\,a_{i+1}\cdots a_4
$$
Bien entendu, pour $n$ quelconque $\widehat a_i = a_1 \cdots a_{i-1} a_{i+1} \cdots a_n$.
En conséquence, on dispose d'une relation homogène de degré $\delta+1$
$$
\sum_{i=1}^n U_iQ_i = 0   \qquad \text{avec}\qquad  U_i = \frac{X^\emouton}{X_i^{d_i-1}}
$$
Vu la définition de $Q_i$, on aurait pu deviner cette relation en écrivant 
$$
\frac{Q_1}{X_1^{d_1-1}} +  \frac{Q_2}{X_2^{d_2-1}} + \cdots + \frac{Q_n}{X_n^{d_n-1}} =
(X_1 - X_2) + (X_2-X_3) + \cdots + (X_n-X_1) = 0
$$
et en chassant les dénominateurs.

\medskip
$\bullet$
Puisque $\uQ$ n'est pas régulière, c'est que le complexe de Koszul
descendant $\rmK_\sbullet(\uQ)$ n'est pas exact. Où et comment?
Réponse: en degré $\delta+1$.  Le complexe composante homogène
$\rmK_{\sbullet,\delta+1}(\uQ)$ n'est pas exact et sa non-exactitude
se manifeste dès le départ en degré homologique 1:
$$
\rmH_1(\uQ_{\delta+1};\bfA) \simeq \bfA
$$
Et surtout: la relation fournie par $(U_1, \dots, U_n)$ est un générateur de ce groupe d'homologie. Note:
ici~$\uQ_{\delta+1}$ désigne la famille constituée de la base monomiale de $\langle\uQ\rangle_{\delta+1}$.
Nous n'allons pas le prouver et nous laissons le soin au lecteur (courageux) de montrer
$$
\Ker\partial_{1,\delta+1}(\uQ) = \Im\partial_{2,\delta+1}(\uQ) \oplus \bfA . \sum_{i=1}^n U_ie_i
$$

\medskip
$\bullet$
Le même lecteur courageux pourra s'interroger sur l'exactitude du complexe $\rmK_{\sbullet,\delta}(\uQ)$
i.e. en degré~$\delta$ au lieu de~$\delta+1$.

$\bullet$
Un dernier argument précisant que $\langle\uQ\rangle$ est de profondeur $n-1$ et pas $n$.
Sur chaque localisé en~$X_\ell$, on a, comme on le voit à la première page du chapitre~\ref{ChapJeuCirculaire},
$\langle\uQ\rangle = \langle X_1-X_2, X_2-X_1, \dots, X_n-X_1\rangle$ qui est de profondeur $n-1$.
Et ces égalités de profondeur se recollent puisque $\Gr(\uX) = n$.

\end {rmqs}  

\subsection{Le format $(1,1,3)$ où $\uomegares:\bfB_\delta\to\bfA$
            est de profondeur $\geqslant 2$ mais non injective}
\label{ExempleProfondeurNonInjective}

Nous avons vu précédemment en~\ref{omegaresGr2} que l'injectivité
de $\uomegares$ en générique 
avait comme conséquence le fait que $\Gr(\omegares) \geqslant 2$. 
Mais la réciproque est fausse. Donnons un exemple simple d'une suite 
régulière $\uP = (P_1, \ldots, P_n)$ vérifiant $\Gr(\omegares) \geqslant 2$ 
sans avoir l'injectivité de $\uomegares$ c'est-à-dire sans avoir 
l'égalité $\langle\uP\rangle_\delta = \Ker(\omegares)$.
On ne peut pas prendre $n=2$ à cause du théorème d'Hilbert-Burch. 
Nous avons misé sur $n=3$. Quant au choix de $D = (d_1,d_2,d_3)$, 
il est guidé par le fait que nous voulons contrôler à la main la non exactitude de 
$$
\xymatrix @C = 1.2cm{ 
\rmK_{1,\delta}\ar[r]^-{\Syl_{\delta}} &
\rmK_{0, \delta} \ar[r]^{\omegares} & \bfA
}
$$
Comme $\rmK_{3,\delta} = 0$  (rappel: en degré $d = \delta$, on a 
toujours $\rmK_{n,\delta} = 0$), 
le rang attendu de $\partial_{2,\delta}$ est $\dim \rmK_{2,\delta}$.
Il va être pratique d'imposer à $\rmK_{2,\delta}$ d'être de petite dimension pour contrôler 
la profondeur de l'idéal des mineurs déterminantiels de rang attendu 
de $\partial_{2,\delta}$. Imposons donc $\dim \rmK_{2,\delta} = 1$.
L'utilisation de la formule:
$$
\dim \rmK_{k,d} = \sum_{\#I = k} \binom{n+d-d_I-1}{d-d_I} 
\qquad \text{ où } d_I = \sum_{i \in I} d_i
$$
nous a conduit au plus petit exemple $D = (1,1,3)$. 
Désignons par $L_1$ et $L_2$ les deux premiers polynômes de 
$\uP = (L_1, L_2, P_3)$ :
$$
L_1 = a_1X_1+ a_2X_2 + a_3X_3, 
\qquad 
\qquad L_2 = b_1X_1+ b_2X_2 + b_3X_3
$$
où les $a_i, b_i$ seront déterminés plus loin ; 
quant à $P_3$, il n'interviendra pratiquement pas.

\medskip
Décrivons le complexe augmenté :
$$
\xymatrix @C = 1.2cm{0 \ar[r] & 
\rmK_{2,\delta}\ar[r]^{\partial_{2,\delta}} & 
\rmK_{1,\delta}\ar[r]^-{\Syl_{\delta}} &
\rmK_{0, \delta} \ar[r]^{\omegares} & \bfA
}
$$
en commençant par étudier le cas de la suite générique.

\bigskip

$\bullet$ \textbf{Première passe avec une suite $\uP$ générique}

\bigskip

D'après l'étude faite à la page~\pageref{evalxiProperties}  
sur le format $D = (1,\dots, 1, e)$, on sait que 
la forme linéaire $\omegares : \bfA[\uX]_\delta \to \bfA$ 
est très simple et est donnée par 
la formule close (cf.~\ref{omegaresCasEcole}) :
$$
\omegares = \evalxi  \qquad
\hbox {où $\uxi=(\xi_1, \xi_2, \xi_3)$ est défini par} \quad
\begin {vmatrix} a_1 &b_1 &T_1 \\ a_2 &b_2 &T_2 \\ a_3 &b_3 &T_3 \end {vmatrix} =
\xi_1 T_1 + \xi_2 T_2 + \xi_3 T_3
$$
formule où l'on voit que le polynôme $P_3$ n'intervient pas.
Comme $\rmK_{0,\delta} = \bfA[X_1, X_2, X_3]_2$, 
matriciellement, cette forme linéaire 
est représentée par 
$$
\omegares \ = \ 
\NorthEastBordermatrix{
\Veti{X_1^2} & \Veti{X_1X_2} & \Veti{X_1X_3} & \Veti{X_2^2} & \Veti{X_2X_3} & \Veti{X_3^2} & \\
\xi_1^2 & \xi_1\xi_2 & \xi_1\xi_3 & \xi_2^2 & \xi_2\xi_3 & \xi_3^2 & \\
}
$$

\medskip

Examinons l'application de Sylvester en degré $\delta = 2$. 
Rappelons que 
pour un format du type $D = (1,\dots, 1, e)$, l'application $\Syl_\delta$ 
attachée au jeu $(L_1, \dots, L_{n-1},P_n)$ 
est donnée par 
$$
\Syl_\delta : 
\begin{array}[t]{rcl}
\rmK_{1,\delta} = \bigoplus\limits_{i=1}^{n-1} \bfA[\uX]_{\delta-1} e_i 
& \longrightarrow & \rmK_{0,\delta} = \bfA[\uX]_\delta  \\ [0.2cm]
X^\gamma e_i & \longmapsto & X^\gamma L_i 
\end{array}
$$
Ici, pour le format $D = (1,1,3)$, on a 
$\rmK_{0,\delta} = \bfA[X_1, X_2, X_3]_2$
et 
$\rmK_{1,\delta}$ a pour $\bfA$-base monomiale les $X_i e_j$ où $j \in \{1,2\}$.
Ces deux modules ont même dimension, à savoir $6$.
Matriciellement, on obtient 
\label{SyldeltaFormat113}
$$
\Syl_\delta \ = \ 
\NorthEastBordermatrix{
\Veti{X_{1}\,e_{1}} & \Veti{X_{2}\,e_{1}} & \Veti{X_{3}\,e_{1}} & \Veti{X_{1}\,e_{2}} & \Veti{X_{2}\,e_{2}} & \Veti{X_{3}\,e_{2}} & \\
a_{1} & . & . & b_{1} & . & . & \Heti{X_{1}^{2}} \\
a_{2} & a_{1} & . & b_{2} & b_{1} & . & \Heti{X_{1}X_{2}} \\
a_{3} & . & a_{1} & b_{3} & . & b_{1} & \Heti{X_{1}X_{3}} \\
. & a_{2} & . & . & b_{2} & . & \Heti{X_{2}^{2}} \\
. & a_{3} & a_{2} & . & b_{3} & b_{2} & \Heti{X_{2}X_{3}} \\
. & . & a_{3} & . & . & b_{3} & \Heti{X_{3}^{2}} \\
}
$$
Le module $\rmK_{2,\delta} = \bfA e_{12}$ est libre de dimension $1$ et la différentielle 
$\partial_{2,\delta}$ est définie par 
$e_{12} \mapsto L_1e_2 - L_2e_1$ donc 
$$
\partial_{2,\delta} 
\ = \ 
\NorthEastBordermatrix{
\Veti{e_{12}} & \\
-b_1 & \Heti{X_1e_1} \\
-b_2 & \Heti{X_2e_1} \\
-b_3 & \Heti{X_3e_1} \\
a_1 & \Heti{X_1e_2} \\
a_2 & \Heti{X_2e_2} \\
a_3 & \Heti{X_3e_2} \\
}
$$
\bigskip

$\bullet$ \textbf{Deuxième passe avec une suite $\uP$ \og moins générique \fg{} !}

\bigskip

Fixons maintenant les $a_i, b_i$, coefficients de $L_1$ et $L_2$, 
en prenant seulement \textit{deux} indéterminées
$a,b$ (sur un petit anneau de base~$\bfk$). 
On fait le choix de :
$$
L_1 = aX_1 + bX_2, \qquad \qquad 
L_2 = bX_1 + aX_3
$$ 
si bien que le complexe augmenté 
$$
\xymatrix @C = 1.2cm{0 \ar[r] & 
\rmK_{2,\delta}\ar[r]^{\partial_{2,\delta}} & 
\rmK_{1,\delta}\ar[r]^-{\partial_{1,\delta}} &
\rmK_{0, \delta} \ar[r]^{\omegares} & \bfA
}
$$
a pour différentielles : 
$$
\NorthEastBordermatrix{
\Veti{e_{12}} & \\
-b & \Heti{X_{1}\,e_{1}} \\
. & \Heti{X_{2}\,e_{1}} \\
-a & \Heti{X_{3}\,e_{1}} \\
a & \Heti{X_{1}\,e_{2}} \\
b & \Heti{X_{2}\,e_{2}} \\
. & \Heti{X_{3}\,e_{2}} \\
}
\qquad 
\NorthEastBordermatrix{
\Veti{X_{1}\,e_{1}} & \Veti{X_{2}\,e_{1}} & \Veti{X_{3}\,e_{1}} & \Veti{X_{1}\,e_{2}} & \Veti{X_{2}\,e_{2}} & \Veti{X_{3}\,e_{2}} & \\
a & . & . & b & . & . & \Heti{X_{1}^{2}} \\
b & a & . & . & b & . & \Heti{X_{1}X_{2}} \\
. & . & a & a & . & b & \Heti{X_{1}X_{3}} \\
. & b & . & . & . & . & \Heti{X_{2}^{2}} \\
. & . & b & . & a & . & \Heti{X_{2}X_{3}} \\
. & . & . & . & . & a & \Heti{X_{3}^{2}} \\
}
\qquad 
\NorthEastBordermatrix{
\Veti{X_1^2} & \Veti{X_1X_2} & \Veti{X_1X_3} & \Veti{X_2^2} & \Veti{X_2X_3} & \Veti{X_3^2} & \\
a^2b^2 & -a^3b & -ab^3  & a^4  &a^2b^2  &b^4 & \\
}
$$
Un petit détail pour la matrice de $\omegares$ à droite. Comme le format est particulier, à savoir de la forme $D=(1,\dots, 1,e)$, on a 
$\omegares = \evalxi$ (cf.~\ref{omegaresCasEcole}).
Ici, on a 
$$
\begin{vmatrix} a &b &T_1 \\ b & . &T_2 \\ . &a &T_3 \end{vmatrix} 
=
ab T_1 - a^2 T_2 -b^2 T_3 
\qquad \text{donc} \quad
\uxi = (ab, -a^2, -b^2)
$$
Ainsi 
$$
\omegares(X^\alpha) = \evalxi(X^\alpha) = 
\uxi^\alpha = \xi_1^{\alpha_1} \xi_2^{\alpha_2} \xi_3^{\alpha_3}
= (ab)^{\alpha_1}(-a^2)^{\alpha_2}(-b^2)^{\alpha_3}
$$

$\rhd$ L'idéal des mineurs d'ordre $r_{2,\delta} = 1$ de $\partial_{2,\delta}$ est 
de profondeur au plus $2$ (en fait, ici exactement $2$ car $a$ et $b$ sont deux indéterminées).
On a donc $\Gr\big(\calD_{r_{2,\delta}}(\partial_{2,\delta})\big) < 3$ et l'utilisation du critère \og 
What makes a complex exact \fg{} empêche le complexe augmenté d'être exact en
$\rmK_{0, \delta}$, autrement dit $\uPdelta \subsetneq \Ker \omegares$.

\medskip

Pour confirmer ce qui vient d'être dit, 
le lecteur pourra vérifier que $F := X_1^2 - X_2X_3$, qui est dans
$\Ker(\omegares)$, n'appartient pas $\uPdelta = \Im \Syl_\delta$. 
En revanche, il s'étonnera peut-être de constater que 
les deux polynômes $aF$ et $bF$ 
sont dans $\Im \Syl_\delta = \uPdelta = \langle L_1, L_2 \rangle_\delta$, puisque :
$$
\left\{ 
\begin {array} {l}
\hbox {col. 1 $-$ col. 5:} \qquad
\Syl_\delta(X_1e_1 - X_2e_2) = X_1L_1 - X_2L_2 = aF \\[0.2cm]
\hbox {col. 4 $-$ col. 3:} \qquad
\Syl_\delta(X_1e_2 - X_3e_1) = X_1L_2 - X_3L_1 = bF \\
\end {array}
\right.
$$

$\rhd$
On a $\Gr(\omegares) \geqslant 2$. Un premier argument réside dans le
fait que les $\omegares(X^\alpha)$ sont, au signe près, des monômes en
$a$ et $b$, premiers dans leur ensemble, donc ici une suite de
profondeur~2. Un autre argument, bien meilleur car plus général,
consiste à remarquer que $\Gr(\uxi) \geqslant 2$ et $\Im\omegares
= \langle\uxi\rangle^2$, égalité provenant de $\omegares = \evalxi$.
Plus général car il s'applique à tout format $D = (1,\dots,1,e)$:
on a $\Im\omegares = \langle\uxi\rangle^\delta$ avec $\delta=e-1$ prouvant
l'implication $\Gr(\uxi) \geqslant 2 \Rightarrow \Gr(\omegares) \geqslant 2$.

$\rhd$ 
Quant à $P_3$, il y a beaucoup de latitude dans le choix d'un polynôme homogène de degré~3 (par exemple, 
un monôme): il suffit juste d'assurer que la suite $\uP = (L_1, L_2, P_3)$ obtenue est régulière où 
$L_1 = aX_1 + bX_2$,  $L_2 = bX_1 + aX_3$.
Choisissons $P_3 = X_1^3$.
Par la propriété d'exponentation des suites régulières
(cf.~\ref{PropProfondeur2}), il suffit de voir que 
la suite $(L_1, L_2, X_1)$ est régulière ; 
et par permutabilité en terrain homogène (cf.~\ref{PermutabiliteSuiteReguliereHomogene})
que la suite $(X_1, L_1, L_2)$ l'est. 
Modulo~$X_1$, cela revient à voir que $(bX_2, aX_3)$ est régulière, ce qui est évident.

\bigskip

$\bullet$ \textbf{Troisième passe avec la suite $(aX_1+bX_2,\ bX_1+aX_2, \ X_1^3)$}

\bigskip

Sur le système $\uP = (L_1,L_2,X_1^3)$ ci-dessus, nous allons apporter
les précisions et compléments ci-dessous permettant d'illustrer un
certain nombre de notions intervenues. Tout peut
être explicité, en partie grâce au fait que ce système relève du premier
cas d'école, cf la section \ref{soussectionPremierCasEcole}. Par
exemple, l'appartenance suivante, due au fait que $\omegares
= \evalxi$, est élémentaire (cf la proposition \ref{evalxiProperties}):
$$
\forall\, F,G \in \bfA[\uX]_\delta,\quad
\omegares(F)G - \omegares(G)F \in \langle\uP\rangle_\delta 
$$
Il en est de même de la détermination de $\calR = \omegares(\nabla)$, 
qui sera défini plus loin comme étant le résultant de $\uP$ (qui engendre l'idéal d'élimination en terrain générique, cf le théorème \ref{Resultant11d}).

On pourra remarquer que l'on dispose d'un automorphisme involutif
$(a \leftrightarrow b, X_2 \leftrightarrow X_3)$ de $\bbZ[a,b][\uX]$
qui échange $L_1$, $L_2$ et fixe $P_3$. On rappelle que $\uxi = (ab,
-a^2, -b^2)$ et que $\delta=2$.

\begin {enumerate}
\item
Les scalaires $a,b$ sont réguliers sans  être $\bfB_\delta$-réguliers.

\item 
Il en est donc de même de $\det W_{1,\delta} = ab^4$ 
(car multiple de $a$, ou de $b$, comme on veut).
  
\item
Détermination de $\calR = a^3b^3$, via $\evalxi(\nabla)$, ou via $P_3(\uxi)$.

\item
Détermination du cofacteur: $\omega = a\,\omegares$.

\item
Détermination de $\uPsat$.  
\end {enumerate}  

En ce qui concerne 1. tout repose sur $F = X_1^2 - X_2X_3$. 
En effet, on a 
$$
aF = 
aL_1 - b L_2 \in \langle\uP\rangle_2
\qquad \text{ mais } \qquad 
F \notin \langle\uP\rangle_2
$$
Idem avec le polynôme $bF$ si l'on veut.
On peut d'ailleurs préciser que $F \notin \uPsat_2$ en constatant que
l'idéal contenu de $F$ est $\langle 1\rangle$ tandis que:
$$
\uPsat_2 \subset (a\bfA+b\bfA)\bfA[\uX]_2
$$
En effet, d'après Wiebe, cf \ref{MiniWiebe}, on a
$\uPsat_2 = \langle\uP\rangle_2 \oplus \bfA\nabla$
avec $\langle\uP\rangle_2 =\langle L_1,L_2\rangle_2 \subset (a\bfA+b\bfA)\bfA[\uX]_2$.
Calculons un déterminant bezoutien $\nabla$ ; on peut considérer l'égalité matricielle :
$$
[L_1, L_2, P_3] = [X_1, X_2, X_3]
\begin {bmatrix}
a & b &X_1^2 \\
b & . &. \\  
. & a &. \\
\end{bmatrix}
\qquad \text{d'où } \qquad
\nabla = abX_1^2  
$$

2. On dresse la matrice $W_{1,\delta}$ :
ici, dans cet exemple, elle correspond à la matrice de $\Syl_\delta$ de la page~\pageref{SyldeltaFormat113}, à laquelle 
on supprime la ligne $X^\emouton = X_3^2$ et la colonne $X_1e_2$.
Un calcul facile montre que $\det W_{1,\delta} = ab^4$.

\medskip

3. Pour déterminer le résultant $\calR = \omegares(\nabla)$, on peut utiliser 
la formule $\evalxi(\nabla)$ (cf.~\ref{omegaresCasEcole}), qui fournit $\evalxi(abX_1^2) =ab \xi_1^2 = (ab)^3$, ou la formule $P_3(\uxi) = \xi_1^3 = (ab)^3$, cf le premier cas d'école~\ref{soussectionPremierCasEcole}.

\medskip

4. Quant au cofacteur, on peut le déterminer avec l'évaluation des formes linéaires en le \MoutonNoir.
On a d'une part $\omega(X_3^2) = \det W_{1,\delta}(\uP) = ab^4$ et
d'autre part $\omegares(X_3^2) = \evalxi(X_3^2) = \xi_3^2 = (-b^2)^2 = b^4$, d'où l'égalité $\omega = a \omegares$.

\bigskip

Le point 5. est le plus difficile. 

\noindent
On a $\uPsat_d = \langle\uP\rangle_d$ pour $d > \delta=2$ et pour $d=2$, on a 
$\uPsat_2 = \langle\uP\rangle_2 \oplus \bfA\nabla$ comme on l'a déjà signalé précédemment. Reste la détermination des composantes
homogènes de degré $0$ et $1$.

En ce qui concerne la composante homogène de degré 0, on ne peut pas directement
invoquer la monogénéité de l'idéal d'élimination (théorème \ref{Resultant11d} du premier cas d'école) 
par \og manque de généricité\fg{} pour affirmer:
$\ElimIdeal = \langle\calR\rangle$.

Il est cependant possible de généraliser l'énoncé du théorème cité de
la manière suivante: si une suite régulière $\uP = (L_1, \dots,
L_{n-1}, P_n)$ de format $(1,\dots, 1,e)$ vérifie $\Gr(\uxi) \geqslant 2$,
alors
$$
\ElimIdeal = \langle \calR \rangle
\quad \text{ où $\calR = \evalxi(\nabla)$}
$$
En effet, on a $\Gr(\evalxi) \geqslant 2$ et on peut alors appliquer la proposition \ref{MonogeneiteElimIdeal} (forme linéaire mirifique).
On en rappelle le principe. Soit $a \in \ElimIdeal$.
Il s'agit de montrer que $a$ est multiple de $\evalxi(\nabla)$.
Pour tout $X^\alpha$ de degré $\delta$, on a:
$$
aX^\alpha \in \uPsat_\delta = \langle\uP\rangle_\delta + \bfA\,\nabla
\qquad \qquad
\text{\small (l'égalité est conséquence de Wiebe, cf \ref{MiniWiebe})}
$$
En appliquant $\evalxi$ qui est nulle sur $\langle\uP\rangle_\delta$ d'après~\ref{evalxiProperties}, on obtient
$$
\forall\ |\alpha| = \delta, \quad \evalxi(\nabla) \mid a\,\evalxi(X^\alpha)
$$
Le fait que $\Gr(\evalxi) \geqslant 2$ permet d'en déduire 
$\evalxi(\nabla) \mid a$, d'où $a \in \langle \calR \rangle$,
ce qu'il fallait démontrer.

\medskip

De manière générale, l'égalité $\ElimIdeal = \langle \omegares(\nabla) \rangle$,
pour un format quelconque et une suite régulière~$\uP$, est valide
lorsque $\Gr(\omegares) \geqslant 2$: c'est l'objet du théorème
ultérieur~\ref{ResGenElimIdeal}.

\bigskip

Enfin, en ce qui concerne la composante homogène de degré 1 de
$\uPsat$, c'est plus fastidieux et il est indispensable d'avoir la
table des $\omegares(X_iX_j)$ sous les yeux. Pour $i$ fixé:
$$
\pgcd\big(\omegares(X_iX_j), j \in \{1,2,3\}\big) =
\begin {cases}
\pgcd(a^2b^2,-a^3b,-ab^3) = ab &\text{si $i=1$} \\
\pgcd(-a^3b,a^4, a^2b^2) = a^2 &\text{si $i=2$} \\
\pgcd(-ab^3,a^2b,b^4) = b^2    &\text{si $i=3$} \\
\end {cases} 
$$
Nous allons en déduire $\uPsat_1 \cap \bfA X_i
= \bfA Y_i$ où les $Y_i$ sont donnés par:
$$
Y_1 = \frac{a^3b^3}{ab} X_1 = a^2b^2\,X_1, \qquad
Y_2 = \frac{a^3b^3}{a^2} X_2 = ab^3\,X_2, \qquad
Y_3 = \frac{a^3b^3}{b^2} X_3 = a^3b\,X_3
$$
Traitons le cas de $X_1$. Soit
$\lambda\in \bfA$ tel que $\lambda X_1 \in \uPsat$ ; on a donc une
appartenance $\lambda X_1X_j \in \langle \uP\rangle_2
+ \bfA\nabla$ pour $j=1,2,3$. Un coup de $\omegares$ donne
$$
\lambda\,\omegares(X_1X_j) \equiv 0  \bmod \omegares(\nabla)
\qquad \text{i.e.} \qquad
a^3b^3 \mid \lambda\,\pgcd\big(\omegares(X_1X_j), j \in \{1,2,3\}\big) = \lambda\,ab
$$
D'où $a^2b^2 \mid \lambda$, d'où $\lambda X_1 \in \bfA Y_1$. 
Reste à contrôler que $Y_1 \in \uPsat$ ; autrement dit, a-t-on pour $j=1,2,3$ 
des égalités du type :
$$
Y_1X_j \overset{?}{=} \text{un élément de $\langle L_1, L_2\rangle_2$} + q_j\nabla \ ;
\quad \text{comme nécessairement} \quad
q_j = \frac{\omegares(Y_1X_j)}{\omegares(\nabla)} = \frac{\omegares(X_1X_j)}{ab}
$$
il n'y a plus qu'à vérifier que $Y_1X_j - q_j\nabla \in \langle L_1,L_2\rangle_2$. C'est bien le cas, car :
$$
Y_1X_1 - ab\nabla = 0, \qquad Y_1X_2 + a^2\nabla = ab(bX_2 -aX_1)L_1 , \qquad
Y_1X_3 + b^2\nabla = ab^2X_1L_2 
$$
Un travail analogue pour $X_2$ fournit $Y_2 = ab^3X_2$ avec
$Y_2X_j \in \langle L_1,L_2\rangle_2 + \bfA\nabla$:
$$
Y_2X_1 + ab\nabla = ab^2 X_1L_1, \qquad
Y_2X_2 - a^2\nabla = ab (bX_2-aX_1) L_1, \qquad
Y_2X_3 - b^2\nabla = -ab^3F
$$
Pour $X_3$, il suffit d'appliquer à $X_2$ l'involution $(a \leftrightarrow b, X_2 \leftrightarrow X_3)$.

\medskip
Ces $Y_i$ sont liés modulo $\uP$ au sens où, pour $i \ne j$, on a $Y_i \pm
Y_j \in \langle L_1,L_2\rangle \subset \langle\uP\rangle$.  Très
précisément:
$$
Y_1 + Y_2 = ab^2\,L_1,\qquad  Y_1 + Y_3 = a^2b\,L_2
$$
On a donc obtenu:
$$
\bfA L_1 \oplus \bfA L_2 \oplus \bfA Y_1 =
\bfA L_1 \oplus \bfA L_2 \oplus \bfA Y_2 =
\bfA L_1 \oplus \bfA L_2 \oplus \bfA Y_3 \ \subseteq \ \uPsat_1
$$
Faisons un dernier effort pour démontrer qu'il y a égalité à droite.  Soit $G
= \gamma X_1 + H \in \uPsat_1$ avec $H = \beta X_2 + \alpha X_3$.
Puisque $Y_1=a^2b^2X_1 \in \uPsat_1$, il en est de même de $a^2b^2G
- \gamma Y_1 = a^2b^2H$ et donc $a^2b^2HX^2$ appartient à $\uPsat_2 = \langle \uP\rangle_2 + \bfA \nabla$. Un coup de $\omegares$ donne
$a^3b^3 \mid \omegares(a^2b^2HX_2)$ donc
$$
ab \mid \omegares(HX_2) = \beta a^4 + \alpha a^2b^2
\quad \text{ d'où } \quad b \mid \beta
$$
De la même manière, en écrivant $a^3b^3 \mid \omegares(a^2b^2HX_3)$, on obtient $a \mid \alpha$.
En conséquence:
$$
G - \frac{\beta}{b}\,L_1 - \frac{\alpha}{a}\,L_2 \ \in\ 
\uPsat_1 \cap \bfA X_1 = \bfA Y_1
$$
En définitive, $G \in \bfA L_1 + \bfA L_2 + \bfA Y_1$.

Bilan. On a l'égalité :
$$
\uPsat_1 = \bfA L_1 + \bfA L_2 + \bfA a^2b^2X_1 =
\bfA L_1 + \bfA L_2 + \bfA ab^3X_2 = \bfA L_1 + \bfA L_2 + \bfA a^3bX_3 
$$

\subsection{Les cofacteurs scalaires $b_d^\sigma(\protect\uP)_{d\geqslant \delta}$ pour $\protect\uP$ quelconque
             et $\sigma \in \fS_n$}
\label{CofacteursTordus}
\index{cofacteur scalaire $b_d(\uP)$}%

Dans les modules/idéaux $\DVect_{s_\delta}(\Syl_\delta)$ et $\calD_{s_d}(\Syl_d)$, nous avons explicité
des éléments privilégiés (confer \ref{sectionFormeLineaireOmega},
\ref{SigmaVersion}, \ref{DefW1}), à savoir
$$
\omega^\sigma 
\qquad \text{ et } \qquad 
\forall\, d \geqslant \delta + 1,\ 
\det W_{1,d}^\sigma
$$
Ils sont multiples de l'invariant de MacRae (vectoriel/scalaire) via un cofacteur $b_d^\sigma$
qui est parfaitement défini 
lorsque la suite $\uP$ est générique. D'où la définition et propriétés suivantes:
\begin{quote}
\it 
Par spécialisation du 
cas générique, on définit pour une suite $\uP$ quelconque et $d \geqslant \delta$,
le cofacteur $b_d^\sigma = b_d^\sigma(\uP) \in \bfA$. Il vérifie :
$$
\omega^\sigma 
\ = \ 
\omegares \, \, b_\delta^\sigma
\qquad \text{ et } \qquad 
\forall\,d \geqslant \delta+1,\ 
\det W_{1,d}^\sigma = \calR_d \, b_d^\sigma
$$
Pour le jeu étalon généralisé, 
$b_d^\sigma(\pXD) 
= \prod\limits_{X^\beta \in \Jex_{2,d}} p_{\sminDiv(X^\beta)}$.
Et le jeu étalon, $b_d^\sigma(\uX^D) = 1$.

Si $\uP$ couvre le jeu étalon généralisé, la composante homogène dominante de $b_d^\sigma(\uP)$ est 
$b_d^\sigma(\pXD)$.
\end{quote}

\label{NOTA10-bdsigma}%
%
%

\begin {prop} [Propriétés des cofacteurs tordus]
\label{ProprietesCofacteurs}

\leavevmode
\begin{enumerate}[\rm i)]
\item 
Pour $d \geqslant \delta$, le cofacteur $b_d = b_d(\uP)$ est homogène en $P_i$ de poids $\dim \Jex_{2,d}^{(i)}$.
Ce qui signifie qu'en notant $\nu_i$ le nombre de monômes de $\Jex_{2,d}$ de $\minDiv$ égal à $i$, on a
pour tous $a_1, \dots, a_n \in \bfA$:
$$
b_d(a_1P_1, \dots, a_nP_n) = a_1^{\nu_1}\cdots a_n^{\nu_n}\, b_d(P_1, \dots, P_n)
$$
En particulier, $b_d$ ne dépend pas de $P_n$.

\item
Plus généralement, pour $\sigma \in \fS_n$, en notant
$\nu_i$ le nombre de monômes de $\Jex_{2,d}$ de $\sminDiv$ égal à $i$:
$$  
b^\sigma_d(a_1P_1, \dots, a_nP_n) = a_1^{\nu_1}\cdots a_n^{\nu_n}\, b^\sigma_d(P_1, \dots, P_n)
$$
En particulier, le scalaire $b^\sigma_d$ ne dépend pas de $P_{\sigma(n)}$.

\item 
Pour $\uP$ couvrant le jeu étalon généralisé, on a 
$$
\Gr\big((b_d^\sigma(\uP))_{\sigma\in\Sigma_2}\big) \geqslant 2
$$
\end{enumerate}
\end {prop}

\begin {proof} \leavevmode

i)
Par définition, on a les égalités de $\bfA$ :
$$
\underbrace{\omega(X^\emouton)}_{\det W_{1,\delta}}
\ = \ 
\omegares(X^\emouton) \, \, b_\delta
\qquad \text{ et } \qquad 
\forall\,d \geqslant \delta+1,\ 
\det W_{1,d}  = \calR_d \, b_d
$$
D'après~\ref{PoidsDetW}, les scalaires à gauche des
égalités sont homogènes en $P_i$ de poids $\dim \Jex_{1,d}^{(i)}$.
Quant à $\omegares(X^\emouton)$ et $\calR_d$, ils sont homogènes en
$P_i$ de poids $\dim \Jex_{1\setminus 2,d}^{(i)}$
(cf.~\ref{PoidsNormalisationMacRae}).  Par conséquent, les
cofacteurs~$b_d$ sont homogènes en $P_i$ de poids $\dim
\Jex_{1,d}^{(i)} - \dim \Jex_{1\setminus 2,d}^{(i)}$.  D'où le
résultat puisque $\Jex_{1,d}^{(i)} = \Jex_{1\setminus 2,d}^{(i)}
\oplus \Jex_{2,d}^{(i)}$.

\smallskip

On peut aussi fournir la preuve suivante (relativement similaire) :
les cofacteurs $b_d$ sont homogènes en $P_i$ 
puisque ce sont des diviseurs d'éléments eux-mêmes homogènes en $P_i$.
Le poids est obtenu en spécialisant en le jeu étalon généralisé.
Et les propositions \ref{dCofacteurJeuEtalonGen} et~\ref{deltaCofacteurJeuEtalonGen}
fournissent le résultat.

\smallskip

Pour le \og En particulier \fg{}, il suffit de remarquer que
$\Jex_{2,d}^{(n)} = 0$ (quel que soit $d$), car un monôme $X^\alpha$
divisible par deux $X_i^{d_i}$ distincts vérifie $\minDiv(X^\alpha) <
n$.  Ainsi, $b_d$ est homogène en $P_n$ de poids $0$, ce qui signifie
que $b_d$ ne dépend pas de $P_n$.

\medskip
ii)
La même preuve s'applique à $\sigma \in \fS_n$ en considérant 
$\sminDiv$ en lieu et place de $\minDiv$.

\medskip
iii)
D'après le contrôle de la profondeur par les composantes homogènes dominantes (cf.~\ref{ControleProfondeur}), 
il suffit de montrer cette propriété sur les composantes homogènes dominantes des~$b_d^\sigma$. 
Or cette composante homogène est le \textit{monôme} en $p_1, \dots, p_n$:
$$
\pi^\sigma \ = \ 
\prod_{X^\beta \in \Jex_{2,d}} p_{\sminDiv(X^\beta)}
$$
Ce produit ne dépend pas de $p_{\sigma(n)}$ d'après i).
Ainsi, pour toute indéterminée $p_i$, il y a, par définition de $\Sigma_2$, 
un
$\sigma \in \Sigma_2$ tel que $p_i$ ne figure pas dans~$\pi^\sigma$.
On obtient donc, pour les composantes homogènes
$(\pi^\sigma)_{\sigma \in \Sigma_2}$, une famille de monômes premiers
dans leur ensemble, donc de profondeur $\geqslant 2$, 
d'après~\ref{MonomesPremiersEntreEux}.
\end {proof}

\subsubsection*{Identification du cofacteur} 
\label{IdentificationCofacteurCommentaire}

Un des défis sera d'identifier précisément ces cofacteurs.
On montrera   que ce cofacteur peut se lire uniquement à l'aide  de la première différentielle,
à savoir l'application de Sylvester :
$$
\forall\, d \geqslant \delta, \quad 
b_d^\sigma \ = \  
\det W_{2,d}^\sigma
$$
Cette formule, extrêmement simple, sera établie une \emph {première
fois} en~\ref{sousSectionMacaulayRecurrence}, à l'arraché (disons comme 
certains auteurs), par récurrence sur $n$ et en utilisant quelques
propriétés des déterminants $\det W_{h,d}$.  Par ailleurs, cette première preuve
pas très glorieuse nécessite d'avoir avec soi le fait que les invariants de MacRae
des~$\bfB'_d$ sont égaux pour $d \geqslant \delta$,
cf.~\ref{MacRaeEqualities}, ce qui n'est pas une évidence!
Signalons également que pour montrer une formule par récurrence sur~$n$,
il faut la connaître (l'avoir devinée?), ce qui n'est pas rien et vérifier
l'hérédité n'explique en rien sa provenance.

\smallskip

Qui parle d'une première preuve sous entend l'existence d'une seconde preuve?
Exact : une autre stratégie pour identifier ce cofacteur consistera à
remarquer que l'invariant de MacRae d'un module librement résoluble
s'identifie au déterminant de Cayley d'une résolution de ce module
(cf.~\ref{LibrementResolubleImpliqueMacRae}). Et nous disposons d'un
moyen de calcul de ce déterminant de Cayley (scalaire ou vectoriel)
via des quotients alternés de déterminants $\det B^\sigma_{k,d}$ issus des différentielles
de la composante homogène de degré~$d$ du complexe de Koszul de $\uP$.
Précisément, on commencera par montrer (cf. les sections
\ref{DecompositionMacaulay}, \ref{SectionDeltakdMacaulay}
et~\ref{sigmaDecompositionMacaulay}) une égalité de nature
\emph {plus complexe que la précédente}:
$$
\forall\, d \geqslant \delta, \quad  
b_d^\sigma \ = \ 
\dfrac{\det B_{2,d}^\sigma \,\det B_{4,d}^\sigma \, \cdots}
{\det B_{3,d}^\sigma \, \det B_{5,d}^\sigma \, \cdots}
$$
Nul besoin de préciser en quoi ce quotient alterné de déterminants est
plus complexe qu'un seul déterminant!  Mais (il y a en fait plusieurs
mais), cette dernière formule est d'une part \og sans surprise\fg{};
d'autre part, cette formule quotient alterné n'est qu'un cas
particulier d'une formule plus générale car valide pour tout $d$ et
tout degré homologique $k$ (et pas seulement pour $k=2$). Pour
comprendre cette dernière phrase, nous devons mentionner une troisième
famille $\big(\Delta^\sigma_{k,d}\big)_{k \ge 1}$ de scalaires qui
englobe les $\big(b^\sigma_d\big)_{d \ge \delta}$ via l'égalité $b^\sigma_d =
\Delta^\sigma_{2,d}$ pour $d\ge \delta$.

\smallskip

Pour simplifier, nous prenons $\sigma = \Id_n$ et visualisons dans le
schéma suivant les trois protagonistes entremêlés dans cette
histoire:
$$
\vcenter{
\xymatrix @R=1.5cm{
  & *+<0.6cm>[o][F]{\big(w_{h,d}\big)_{h\ge 1}}\ar@{-->}[dl]_-{\textstyle 3}\ar[dr]^{\textstyle 2}
\\
*+<0.6cm>[o][F]{\big(\Delta_{k,d}\big)_{k\ge 1}}\ar@/^15pt/[rr]^{\textstyle 1}
  &&
    *+<0.6cm>[o][F]{\big(\beta_{k,d}\big)_{k\ge 1}}\ar@/^15pt/[ll]^{\textstyle 1'}  
\\  
}}
\qquad
\begin {array}{l}
w_{h,d}\overset{\rm def.}{=}\det W_{h,d},\quad \beta_{k,d} \overset{\rm def.}{=}\det B_{k,d}
\\[0.2cm]
w_{n+1,d} = \beta_{n+1,d} = \Delta_{n+1,d} = 1
\\[0.5cm]
\beta_{k,d} \overset{\textstyle 1}{=} \Delta_{k,d}\Delta_{k+1,d}
\\[0.5cm]
\Delta_{k,d} \overset{\textstyle 1'}{=}
\dfrac{\beta_{k,d}\ \beta_{k+2,d}\ \beta_{k+4,d}\cdots}
      {\beta_{k+1,d}\beta_{k+3,d}\beta_{k+5,d}\cdots}
\\
\end {array}
$$
Essayons d'en expliquer la lecture. Au tout début (cf la section
\ref{SectionDeltakdMacaulay}), nous disposerons de la flèche~1 et
de l'égalité~1, égalité exprimant un scalaire figurant à l'extrêmité
de la flèche en fonction de la famille située en son origine.  Le \og
sans surprise\fg{} plus haut consiste à inverser cette égalité~1 pour
obtenir l'égalité~1'.  Ceci reste assez banal. Ce qui l'est beaucoup
moins est la flèche~2, objet du chapitre~\ref{ChapBW}: nous montrerons
que $\det B_{k,d}$ est un produit pondéré des $(\det W_{h,d})_{h\ge k}$,
les exposants étant des coefficients binomiaux.  Il sera alors
facile d'en déduire la flèche 3 en pointillés, censée schématiser le fait
que $\Delta_{k,d}$ s'exprime en fonction des $\big(\det W_{h,d}\big)_{h\ge k}$.  Ce
qui conduira, pour $k=2$, à la formule extrêmement simple
$\Delta_{2,d} = \det W_{2,d}$ pour tout~$d$ dont l'expression du
cofacteur $b_d = \det W_{2,d}$ est un cas particulier pour~$d \ge
\delta$.

\smallskip

L'étude du déterminant de Cayley débutera au chapitre~\ref{ComplexeDecompose}
et ses relations plus poussées avec les déterminants excédentaires feront l'objet du
chapitre~\ref{ChapBW}, cf en particulier les
théorèmes~\ref{DetBkdProdBinomialDetWhd} et~\ref{DeltakdProdBinomialDetWhd}.

\subsection{Pgcd fort pour un système $\protect\uP$ couvrant le jeu étalon généralisé }

Soit $M$ un module de MacRae de rang $0$. Par définition, son
invariant de MacRae $\MacRae(M)$ est le pgcd fort de tout système
générateur fini de $\calF_0(M)$; on peut prendre par exemple les
mineurs pleins (au sens des lignes) de n'importe quelle présentation de $M$.

Il en est de même pour un module $M$ de MacRae de rang~1. Pour toute
présentation $u : E \to F$ de $M$, nécessairement de rang $r = \dim F
- 1$, son invariant de MacRae $\MacRaeVect(M)$ est, par définition, le
sous-module de $M^\star$ de base une forme linéaire pgcd fort des
formes linéaires du sous-module~$\DVect_r(u)$ de~$F^\star$,
celles-ci étant vues sur $M = \Coker(u)$ par passage au quotient.

\smallskip

Dans le
cadre du résultant, cette vision de l'invariant de MacRae présente
plusieurs inconvénients : d'une part, le nombre relativement élevé de
mineurs, et d'autre part leur manque de maniabilité.  Ici, en degré 
$d \geqslant \delta+1$, nous allons exhiber $n$ mineurs particuliers, qui certes ne
forment pas un système générateur de $\calF_0(\bfB_d)$, mais sont
tels que $\MacRae(\bfB_d)$ continue à être leur pgcd fort. Ce sont des
déterminants de Macaulay, moulés sur le modèle $\det W_{1,d}(\uP)$, et
qui présentent vis-à-vis de la spécialisation certaines propriétés de
robustesse.

\smallskip 

Il convient cependant d'imposer à $\uP$ une contrainte: celle de
couvrir le jeu étalon généralisé. Nous en rappelons le sens
(cf.~\ref{JeuCouvrantEtalon}): l'anneau de base $\bfA$ est de la forme
$\bfA = \bfR[p_1, \dots, p_n]$ où $(p_1, \dots, p_n)$ sont
algébriquement indépendants sur $\bfR$, l'indéterminée $p_i$ étant
allouée au coefficient en~$X_i^{d_i}$ de $P_i$. Si cette contrainte de
couverture du jeu étalon généralisé n'est pas respectée, le
comportement des déterminants de Macaulay peut s'avérer chaotique,
l'exemple paradigmatique étant celui, pour $n=3$, de $\uP = (aY, bX,
cZ^3)$, cf.~\ref{SuiteTypiqueYXZ3}.  Rappelons également qu'un système
$\uP$ couvrant le jeu étalon généralisé est une suite régulière, ce
qui légitime la notion d'invariant de MacRae de~$\bfB_d = \bfA[\uX]_d
/ \langle \uP \rangle_d$.

\smallskip

Voici l'énoncé précis dans lequel nous notons $\Sigma_2$
l'ensemble des transpositions $(i,n)$ pour $1\leqslant i\leqslant n$; ou plus
généralement n'importe quelle partie $\Sigma_2 \subset \fS_n$ telle que
$\Sigma_2(n) = \{1..n\}$.

\begin{prop}[MacRae en termes de pgcd fort de mineurs de Macaulay] 
\label{MacRaePGCDfort}
\leavevmode

Soit $\uP$ un système couvrant le jeu étalon généralisé.

\begin{enumerate}[\rm i)]
\item 
Pour $d \geqslant \delta+1$, l'invariant de MacRae $\calR_d$, générateur
privilégié de l'idéal $\MacRae(\bfB_d)$, est le pgcd fort des mineurs $(\det
W_{1,d}^\sigma)_{\sigma \in \Sigma_2}$.

\item 
La forme linéaire $\omega_{\res} : \bfA[\uX]_\delta \to \bfA$ est le pgcd
fort des formes linéaires $(\omega^\sigma)_{\sigma \in \Sigma_2}$.

\item 
Le scalaire $\omega_{\res}(\nabla)$, générateur privilégié de
$\MacRae(\bfB'_\delta)$, est le pgcd fort des 
$(\omega^\sigma(\nabla))_{\sigma \in \Sigma_2}$.
\end{enumerate}
\end {prop}

\begin {proof} 
Prouvons i) et ii) simultanément.

Ci-dessous, les notions de composantes homogènes sont relatives à la structure
de l'anneau de polynômes $\bfA = \bfR[p_1, \dots, p_n]$. Notons $b_d^\sigma$
les scalaires cofacteurs définis précédemment en~\ref{CofacteursTordus}
$$
\omega^\sigma 
\ = \ 
\omegares \, \, b_\delta^\sigma
\qquad \text{ et } \qquad 
\forall\, d \geqslant \delta+1, \ 
\det W_{1,d}^\sigma 
\ = \ 
\calR_d \, b_d^\sigma
$$
D'après~\ref{ProprietesCofacteurs}, comme $\uP$ couvre le jeu étalon 
généralisé, on a 
$\Gr\big((b_d^\sigma)_{\sigma\in\Sigma_2}\big) \geqslant 2$.

iii) Puisque 
$$
\omega^\sigma (\nabla) \ = \  \omegares(\nabla) \, \, b_\delta^\sigma
\qquad \hbox {et} \qquad 
\Gr\big(({b_\delta^\sigma})_{\sigma \in \Sigma_2}\big) \geqslant 2
$$
il suffit de montrer que $\omegares(\nabla)$ est régulier dans $\bfA$.
Mais $\uP$ étant une suite régulière (cf.~\ref{JeuCouvrantEtalonRegularite}),
$\omegares(\nabla)$ est l'invariant de MacRae de $\bfB'_\delta$,
a fortiori régulier, cf.~\ref{Rem-omegares-regulier}.

\medskip\noindent
Variante. Il suffit de montrer que $\omega(\nabla)$ est régulier dans $\bfA$.
Examinons d'abord le cas du jeu étalon généralisé 
$\pXD = (p_1X_1^{d_1}, \dots, p_nX_n^{d_n})$.
Une matrice bezoutienne est $\diag(p_1X_1^{d_1-1}, \dots, p_nX_n^{d_n-1})$ de déterminant 
$\nabla =
p_1\cdots p_n X^\emouton$ de sorte que
$$
\omega(\nabla) 
\ =\ 
p_1\cdots p_n\, \omega(X^\emouton) 
\ = \ 
p_1 \cdots p_n\, \det(W_{1,\delta}) 
\ = \ 
p_1^{1 + \nu_1} \cdots p_n^{1 + \nu_n}  \qquad \hbox {avec} \qquad
\nu_i = \dim\Jex_{1,\delta}^{(i)}
$$
Supposons maintenant que $\uP$ couvre le jeu étalon généralisé. Alors, pour ce jeu $\uP$,
le scalaire $\omega(\nabla)$, vu dans l'anneau de polynômes $\bfR[p_1, \dots, p_n]$,
a pour composante homogène dominante $p_1^{1 + \nu_1} \cdots p_n^{1 + \nu_n}$ où les
$\nu_i$ sont les $\nu_i$ précédents (la dimension de $\Jex_{1,\delta}^{(i)}$ n'a pas changé
entre temps!). 
On en déduit que $\omega(\nabla)$ est régulier dans $\bfA$. Il en est donc de même
de $\omegares(\nabla)$.
\end {proof}

\begin {rmq}
En degré $d \geqslant \delta+1$, lorsque $\Jex_{2,d} = 0$, le pgcd fort est très facile à déterminer en 
vertu de~\ref{J2dNul} ; 
on a alors $\calR_d = \det W_{1,d} = \det W_{1,d}^\sigma$ pour n'importe
quel $\sigma \in \fS_n$. Idem en degré $\delta$: si $\Jex_{2,\delta} = 0$, on a alors
$\omegares = \omega = \omega^\sigma$. La proposition~\ref{CasParticuliers}, qui
reprend~\ref{J2deltaNul}, donne les formats $D$ pour lesquels cette remarque s'applique.
\end{rmq}

\subsection {Un exemple élémentaire de détermination de $\omegaRes{\protect\uP}$ : le format $(1,p,q)$}

En guise d'application du contrôle du poids de
$\omegares$, nous allons, dans le cas du format $D = (1,p,q)$,
déterminer (partiellement) la forme linéaire $\omegares$. De manière précise, pour tout
système $\uP = (P_1, P_2, P_3)$ de format $D$, en notant $p_1$ le
coefficient de $X_1$ dans $P_1 = p_1X_1+ q_2X_2 + r_3X_3$, nous
allons prouver l'égalité
$$
\omega = p_1^e\,\omegares  \leqno(\star)
$$
égalité dans laquelle l'exposant $e$ possède plusieurs expressions dimensionnelles (cf le tableau ci-dessous)
$$
e = \left\{
\begin {array} {l}
\dim \Jex_{2,\delta}^{(1)} = \binom {p-1}{2} + \binom {q-1}{2} 
\\[2mm]
\dim \Jex_{1,\delta}^{(1)} - \dim \Jex_{1\setminus2,\delta}^{(1)} =
\tfrac{\delta(\delta+1)}{2} - (pq-1)
\\
\end {array}
\right.
$$
Pour établir $(\star)$, nous allons commencer par montrer que le
cofacteur $b = b_\delta(\uP) \in \bfA$, intervenant dans l'égalité
$\omega_\uP = b\, \omegaRes{\uP}$, ne dépend pas des coefficients de
$P_2$ et $P_3$ (bien qu'il en dépende a priori).  Puisque $d_2 + d_3 =
p+q$ est strictement supérieur à $\delta =(p-1) + (q-1)$, l'ensemble
de divisibilité d'un monôme de $\Jex_{2,\delta}$ ne peut pas contenir
à la fois $2$ et $3$. \'Etant de cardinal au moins $2$, il vaut donc
ou bien $\{1,2\}$ ou bien $\{1,3\}$.  Le minimum de chacun de ces
ensembles étant~$1$, on a $\dim \Jex_{2,\delta}^{(2)} = 0$ et $\dim
\Jex_{2,\delta}^{(3)} = 0$.

Une fois cette indépendance acquise, on peut se permettre de prendre
$P_2 = X_2^p$ et $P_3 = X_3^q$ pour déterminer $b$ ; autrement dit, on
a $b = b(\uQ)$ où $\uQ = (P_1, X_2^p, X_3^q)$.  Or le scalaire
$\omega_\uQ(X^\emouton) = \det W_{1,\delta}(\uQ)$ ne dépend pas des coefficients $q_2, r_3$
de $P_1$ (ceci se voit en examinant la matrice $W_{1,\delta}(\uQ)$
dans base monomiale de $\Jex_{1,\delta}$ rangée de manière
décroissante relativement à l'ordre lexicographique pour lequel $X_1 >
X_2 > X_3$: elle est triangulaire inférieure avec une diagonale
contenant soit $p_1$, soit $1$).  Ainsi, $\omega_\uQ(X^\emouton)$ ne
dépend que de $p_1$ et l'exposant sur $p_1$ est donné
par~\ref{PoidsDetW} :
$$
\omega_\uQ(X^\emouton) = p_1^{\dim \Jex_{1,\delta}^{(1)}}
$$
Ainsi $b(\uQ)$ ne dépend que de $p_1$ et d'après~\ref{ProprietesCofacteurs}, on maîtrise 
l'exposant : $b(\uQ) = p_1^{\dim \Jex_{2,\delta}^{(1)}}$.
On a donc $b = p_1^e$ avec $e = \dim \Jex_{2,\delta}^{(1)}$.

Justifions également les autres formules annoncées pour $e$.
Pour cela, établissons le tableau dimensionnel suivant 
dans lequel la dernière ligne est la somme
des deux premières:
$$
\begin {array}{c|c|c|c}  
i  &1  &2 &3 \\[0.5mm]
\hline\vrule height15pt depth5pt width0pt
\dim \Jex_{1\setminus2,\delta}^{(i)} &pq-1        &q-1       &p-1  \\[0.5mm]
\hline \vrule height15pt depth5pt width0pt
\dim \Jex_{2,\delta}^{(i)} &\binom{p-1}{2}+\binom{q-1}{2}  &0       &0  \\[0.5mm]
\hline\vrule height15pt depth5pt width0pt
\dim \Jex_{1,\delta}^{(i)} &\tfrac{\delta(\delta+1)}{2}        &q-1       &p-1  \\[0.5mm]
\end {array}
$$
La première ligne est conséquence du résultat général $\dim\Jex_{1\setminus2,\delta}^{(i)} = \widehat{d}_i-1$
avec ici $D = (d_1, d_2, d_3) = (1,p,q)$.
Dans la deuxième ligne, il reste à établir la dimension 
de $\Jex_{2,\delta}^{(1)}$. 
Les
monômes de $\Jex_{2,\delta}$ d'ensemble de divisibilité $\{1,2\}$ sont
les
$$
X_1^{1+i} X_2^{p+j} X_3^\bullet  \qquad \hbox {avec} \quad
i \geqslant 0,\ j \geqslant 0,\  i+j < q-2  \qquad\qquad
\begin {array}{l}
\hbox {la dernière inégalité provenant de} \\
1+i + p+j \leqslant \delta = p-1 + q-1  \\
\end {array}
$$
Ils sont donc en nombre $\binom{q-1}{2}$. 
De la même façon, les monômes de $\Jex_{2,\delta}$
d'ensemble de divisibilité $\{1,3\}$ sont en nombre $\binom{p-1}{2}$.

La dernière ligne s'obtient en faisant la somme des deux premières.
Mais il ne nous semble pas inutile de déterminer à la main cette dernière ligne.
Commençons par $\dim \Jex_{1,\delta}^{(2)}$ \idest{} le nombre de
$X^\alpha \in \Jex_1$ de degré $\delta = p-1 + q-1$ vérifiant
$\minDiv(X^\alpha) = 2$ donc $\alpha_1 = 0$ et $\alpha_2 \geqslant p$. Un
tel monôme est déterminé par $\alpha_3 = \delta - (\alpha_1+ \alpha_2)$ 
qui doit donc vérifier
$0 \leqslant \alpha_3 \leqslant \delta - p = q-2$. 
Ce qui nous donne les monômes 
$X_2^{\delta-\alpha_3} X_3^{\alpha_3}$ pour $0 \leqslant \alpha_3 \leqslant q-2$ ; donc $\dim \Jex_{1,\delta}^{(2)} = q-1$.
De manière analogue, on a $\dim \Jex_{1,\delta}^{(3)} = p-1$.

Quant aux monômes de $\Jex_{1,\delta}^{(1)}$, ce sont les
monômes vérifiant $\alpha_1 \geqslant 1$, 
ou encore les monômes du type 
$X_1^{\delta-\alpha_2-\alpha_3} X_2^{\alpha_2} X_3^{\alpha_3}$ avec 
$\alpha_2 + \alpha_3 \leqslant \delta -1$ ; 
un petit calcul de dénombrement fournit 
$\dim \Jex_{1,\delta}^{(1)} = \tfrac{\delta(\delta+1)}{2}$. 
On peut également l'obtenir par dimension complémentaire:
$$
\begin {array} {lll}
\dim \Jex_{1,\delta}^{(1)}
 &=& \dim \Jex_{1,\delta} -
      \dim \Jex_{1,\delta}^{(2)} - \dim \Jex_{1,\delta}^{(3)} \\ [0.2cm]
&=&
      \big(\dim \bfA[\uX]_\delta -1\big) \ - \ (q-1) - (p-1) \\ [0.2cm]
&=&
\frac{(\delta+1)(\delta+2)}{2} - 1 \ - \ \delta \\ [0.2cm]
&=& \frac{\delta(\delta+1)}{2}
\\
\end {array}
$$

\subsubsection{Une autre analyse en faisant intervenir certains cofacteurs $b_\delta^\sigma$, $\sigma\in\fS_3$}

Toujours dans le cadre du format $(1,p,q)$, pour n'importe
quelle permutation $\sigma \in \fS_3$, on dispose du cofacteur
$b_\delta^\sigma \in \bfA$, polynôme en les coefficients des $P_i$:
$$
\omega^\sigma = b_\delta^\sigma\,\omegares
$$
Il ne faudrait pas croire que $b_\delta^\sigma$ soit un polynôme aussi
simple que celui trouvé pour $\sigma = \Id_3$; et d'autre part, qu'il
soit facile à déterminer.  Nous fournirons dans le
chapitre~\ref{ChapBW} une formule déterminantale
pour~$b_\delta^\sigma$ (relire à ce sujet le paragraphe \og
Identification du cofacteur\fg{} à la
page~\pageref{IdentificationCofacteurCommentaire}) : il s'agit du
déterminant que nous avons noté $\det W_{2,\delta}^\sigma(\uP)$
(cf. \ref{NOTA05-W2d} ainsi que la définition~\ref{DefWh}):
$$
b_\delta^\sigma(\uP) = \det W_{2,\delta}^\sigma(\uP)
$$
Par exemple, pour $(p,q)=(3,4)$ et $\sigma = (1,3)$, la matrice $W_{2,\delta}^\sigma$,
qui ne dépend pas de $P_{\sigma(3)} = P_1$, est:
$$
W_{2,\delta}^\sigma =
\EastBordermatrix{
p_2 & b_{4} & . & c_{7} & \Heti{X_{1}^{2}X_{2}^{3}} \\ 
. & p_2 & . & c_{11} & \Heti{X_{1}X_{2}^{4}} \\ 
. & b_{8} & p_2 & c_{12} & \Heti{X_{1}X_{2}^{3}X_{3}} \\ 
. & . & b_{10} & p_3 & \Heti{X_{1}X_{3}^{4}} \\ 
}
\qquad\qquad
\det W_{2,\delta}^\sigma = p_2(p_2^2p_3 - p_2b_{10}c_{12} + b_8b_{10}c_{11})
$$
où les $b_i, c_j$ sont les coefficients de $P_2, P_3$:
$$
\left\{
\begin {array}{lcl} 
P_{2} &=& p_2X_{2}^{3} \quad +\quad   b_{4}X_{1}X_{2}^{2} + b_{8}X_{2}^{2}X_{3} + b_{10}X_{3}^{3} + \cdots
\\[0.2cm]
P_{3} &=& p_3X_{3}^{4} \quad +\quad c_{7}X_{1}X_{2}^{3} +  c_{11}X_{2}^{4} + c_{12}X_{2}^{3}X_{3} + \cdots
\\
\end {array}
\right.
$$
Cependant, nous allons voir comment utiliser l'ordre $<_\sigma$ pour
\emph {retrouver directement} le fait que le cofacteur~$b_\delta$ (pour $\sigma = \Id_3$) ne
dépend que de $P_1$. Comme d'habitude, désignons par $p_i$ le
coefficient en $X_i^{d_i}$ de $P_i$ pour $D = (d_1, d_2, d_3) =
(1,p,q)$:
$$
P_1  = p_1 X_1 + \cdots, \qquad  P_2 = p_2X_2^p + \cdots,\qquad P_3 = p_3X_3^q + \cdots
$$
Dans le tableau ci-dessous, pour chacune des permutations
$\sigma \in \fS_3$, on a indiqué l'ordre $<_\sigma$, la valeur
de~$\sigma(3)$ et celle de $b_\delta^\sigma$ quand celle-ci est
susceptible d'être déterminée.  Nous allons expliquer pourquoi ces
permutations sont regroupées en seulement 4 classes séparées par un
trait épais et pourquoi, dans chaque classe, les invariants
$\omega^\sigma$ et $b_\delta^\sigma$ sont les mêmes. Ainsi dans le tableau,
le tiret \og \hbox {---} \fg{} désigne la valeur figurant dans la case au dessus.

Le premier point capital déjà évoqué réside dans le fait que, ni
$\{2,3\}$, ni $\{1,2,3\}$ ne sont des ensembles de $D$-divisibilité en
degré $\delta$, ceci étant dû au fait que $d_2 + d_3 > \delta$. Il
reste donc, en ce degré, comme ensembles possibles de $D$-divisibilité les
singletons, $\{1,2\}$ et $\{1,3\}$.  Puisqu'un objet tordu par
$\sigma$ est piloté par la primitive $\sminDiv$, deux permutations
$\sigma$, $\sigma'$ fournissent le même objet tordu en degré
$\delta$ lorsque
$$
\min_\sigma(E) = \min_{\sigma'}(E) \quad
\begin {array}{c}
\hbox {pour tout ensemble $E$} \\
\hbox {de $D$-divisibilité en degré $\delta$} \\
\end {array}
\quad\hbox{i.e.}\quad
\left\{
\begin {array}{c}
\min_\sigma \{1,2\} = \min_{\sigma'} \{1,2\} \\
\min_\sigma\{1,3\} = \min_{\sigma'} \{1,3\} \\
\end {array}
\right.
$$
On peut ainsi vérifier que $\Id_3$ et $(2,3)$ fournissent le même objet tordu en degré $\delta$. Ainsi $b_\delta^{\Id_3} = b_\delta^{(2,3)}$.
Idem pour les permutations $(1,3)$ et $(1,2,3)$ qui sont regroupées.
$$
\def\saut{\vrule height15pt depth5pt width0pt}
\begin {array}{c|c|c|c}  
\sigma  &\scriptstyle{\sigma(1)<\sigma(2)<\sigma(3)} &\sigma(3) & b_\delta^\sigma \\[0.5mm]
\hline\saut
\Id_3 & 1 < 2 < 3 & 3 &p_1^{\binom{p-1}{2} + \binom{q-1}{2}}  \\[0.5mm]
\hline\saut
(2,3) & 1 < 3 < 2 & 2 & \hbox {---}  \\[0.5mm]
\noalign {\hrule height1.5pt}\saut
(1,3) & 3 < 2 < 1 & 1 & \hbox {expression complexe indépendante de $P_1$}  \\[0.5mm]
\hline\saut 
(1,2,3) & 2 < 3 < 1 & 1 & \hbox {---}  \\[0.5mm]
\noalign {\hrule height1.5pt}\saut
(1,2)  & 2 < 1 < 3 & 3 &p_1^{\binom{p-1}{2}} p_2^{\binom{q-1}{2}}  \\[0.5mm]
\noalign {\hrule height1.5pt}\saut
(1,3,2)  & 3 < 1 < 2 & 2 &p_1^{\binom{q-1}{2}} p_3^{\binom{p-1}{2}}
\end {array}
$$
Ce tableau nous permet de conclure de la manière suivante. On a vu en
\ref{ProprietesCofacteurs} que $b_\delta^\sigma$ ne dépend pas de~$P_{\sigma(n)}$. 
Ainsi $b_\delta^{\Id_3}$ est indépendant de $P_3$, et $b_\delta^{(2,3)}$ est
indépendant de $P_3$. Comme ces deux cofacteurs sont égaux, on obtient ainsi d'une autre manière que $b_\delta$ ne dépend que de $P_1$. Sa détermination précise a été
réalisée précédemment (on a $b_\delta = p_1^e$) ; 
celles de $b_\delta^\sigma$ pour $\sigma =
(1,2)$ et $\sigma = (1,3,2)$ sont laissées au lecteur.

\subsection {\'Etude de la forme $\omegares$ de certains systèmes combinatoires
              $\bsp\protect\uX^D + \bsq\protect\uR$}
\label{SectomegaresJeuxSimplex}

Chaque système $\uQ = (Q_1, \dots, Q_n)$ considéré ici
est défini à partir d'un système monomial $\uR = (R_1, \dots, R_n)$
de format $D$ où $R_i$ est un monôme de degré $d_i$, et à partir de $2n$ indéterminées sur
$\bbZ$ notées $p_1, \dots, p_n, q_1, \dots, q_n$ 
$$
\uQ = \pXD + \bsq\,\uR,\qquad    Q_i \ =\ p_i X_i^{d_i} + q_i R_i
$$
En écrivant $\bbZ[p_1, \dots, p_n, q_1, \dots, q_n] = \bfR[p_1, \dots, p_n]$
avec $\bfR = \bbZ[q_1, \dots, q_n]$, ce jeu $\uQ$ couvre le jeu étalon
généralisé $\pXD = (p_1X_1^{d_1}, \dots, p_nX_n^{d_n})$.

\medskip
Dans la suite, nous \emph{supposons que $\uR$ possède une fonction hauteur} (notée~$h$)
au sens de la section~\ref{SectionJeuxSimples}, l'exemple
typique étant le jeu circulaire généralisé (pour ce dernier jeu,
c'est plutôt $\uQ$ qui est mis en avant au lieu de $\uR$). Pour un tel $\uR$, rappelons
l'objet du théorème \ref{qContributionTheorem} concernant la forme linéaire 
$\omega_\uQ$ (attachée à $\minDiv$) et la formule qui le précède 
portant sur son évaluation en chaque
monôme~$X^\alpha$ de degré~$\delta$.  En désignant par $\calO_\alpha$
l'\og orbite\fg{} (pilotée par $\minDiv$) de $X^\alpha$ sous $\uR$ et
par $\overline{\calO_\alpha}$ son complémentaire
dans~$\Jex_{1,\delta}$, cette évaluation est un monôme en les
$p_i, \overline{q}_j$ (où $\overline{q}_j = -q_j$), produit d'une
$q$-contribution et d'une $p$-contribution:
$$
\forall\, |\alpha| = \delta, \quad 
\omega_\uQ(X^\alpha) \ = \ 
\prod_{X^\beta \in \calO_\alpha \setminus \{X^\emouton\}} \kern -9pt \overline{q}_{\minDiv(X^\beta)}
\prod_{X^\gamma \in \overline{\calO_\alpha}} p_{\minDiv(X^\gamma)}
$$
Cette formule combinatoire de $\omega_\uQ(X^\alpha)$ est capitale dans ce qui suit.
Elle fournit en particulier le fait que $\omega_\uQ(X^\emouton)$ est
le \textit{monôme} suivant de $\bbZ[p_1, \dots, p_n]$ (on peut
également utiliser la proposition~\ref{detW1deltaQ})
$$
\omega_\uQ(X^\emouton) \overset{\rm toujours}{=} \det W_{1,\delta}(\uQ) =
\prod_{X^\beta \in \Jex_{1,\delta}} p_{\minDiv(X^\beta)}
\qquad \hbox {aussi égal à} \qquad
\det W_{1,\delta}(\pXD) 
$$

Dans l'énoncé du théorème \ref{qContributionTheorem}, nous avons vu que la
$q$-contribution peut se lire sur le graphe de~$\uR$, en particulier
est \emph{intrinsèque} à $X^\alpha$ i.e. ne dépend pas du mécanisme
$\minDiv$ qui pilote $\omega_\uQ$. Ici nous apportons un éclairage sur
cette $q$-contribution en faisant intervenir la forme
$\omegaRes{\uQ}$.

\begin {theo} [Formule de Macaulay dans le contexte ci-dessus]\leavevmode
\label{omegaresJeuSimple}

\begin {enumerate} [\rm i)]
\item
Le cofacteur $b_\delta = b_\delta(\uQ)$ intervenant dans $\omega_\uQ = b_\delta\,\omegaRes{\uQ}$
se détermine comme suit:
$$
 b_\delta = \det W_{2,\delta}(\pXD) =  \prod_{X^\beta\in\Jex_{2,\delta}} p_{\minDiv(X^\beta)}  
$$
Il est indépendant des $q_i$.
 
\item
En le \MoutonNoir{} $X^\emouton$, la valeur de $\omegaRes{\uQ}$ est donnée par
$$  
\omegaRes{\uQ}(X^\emouton) = \prod_{X^\beta \in \Jex_{1\setminus2,\delta}} p_{\minDiv(X^\beta)} 
\ =\ 
p_1^{\widehat{d_1}-1} \cdots  p_n^{\widehat{d_n}-1}
$$  
\item
 La détermination  de $\omegaRes{\uQ}(X^\alpha)$ se réalise à partir de celle de $\omega_\uQ(X^\alpha)$
 \emph{sans} la connaissance explicite de $b_\delta$ (mis à part le fait que celui-ci ne dépend
 que des $p_i$):
$$
\omega_\uQ(X^\alpha) = p_1^\sbullet \cdots p_n^\sbullet \ \overline{q}_1^{m_1}\cdots \overline{q}_n^{m_n}
\quad \Longrightarrow\quad
\omegaRes{\uQ}(X^\alpha) =
p_1^{\widehat{d_1}-1-m_1} \cdots p_n^{\widehat{d_n}-1-m_n} \ \overline{q}_1^{m_1}\cdots \overline{q}_n^{m_n}
$$
A fortiori $(m_1, \dots, m_n)$ ne dépend que de $\omegaRes{\uQ}(X^\alpha)$ et pas du mécanisme $\minDiv$. 

\item 
On rappelle que $\ell_i = \exposant(X_i^{d_i}) - \exposant(R_i)$.
Dans tout chemin du graphe de $\uR$ reliant $X^\alpha$ au \MoutonNoir{} $X^\emouton$, le
nombre d'étiquettes~$i$ est égal à $m_i$. De plus, la suite $(m_i)_{1 \leqslant i \leqslant n}$
correspond à l'écriture positive minimale de $\alpha - (\emouton)$ sur les $\ell_i$
et contribue à la définition de la fonction hauteur $h$:
$$
\alpha - (\emouton) = \sum_i m_i \ell_i, \qquad\quad
h(\alpha) = \sum_i m_i
$$
\end {enumerate}  
\end {theo}

\begin{proof}

Commençons par montrer i) et ii).
On sait que 
$\omega_\uQ(X^\emouton)$ est un \textit{monôme} en $p_1, \dots, p_n$.
En tant que diviseurs, on en déduit que $b_\delta(\uQ)$ et $\omegaRes{\uQ}(X^\emouton)$
sont nécessairement, au 
signe près, des monômes en $p_1, \dots, p_n$.  
Nous contrôlons,
au sens de l'anneau $\bfR[p_1,\dots,p_n]$, leur
composante homogène dominante: il s'agit 
respectivement de $\prod\limits_{X^\beta \in \Jex_{2,\delta}} p_{\minDiv(X^\beta)}$ et 
$\prod\limits_{X^\beta \in \Jex_{1\setminus2,\delta}} p_{\minDiv(X^\beta)} =
p_1^{\widehat{d_1}-1} \cdots  p_n^{\widehat{d_n}-1}$.
Donc $b_\delta(\uQ)$ et $\omegaRes{\uQ}(X^\emouton)$
sont exactement ces monômes.

\smallskip
Une variante possible, une fois obtenue l'indépendance en les $q_i$, consiste
à réaliser la spécialisation $q_i := 0$ et à invoquer les résultats
du jeu étalon généralisé (chapitre~\ref{ChapJeuEtalonGeneralise}).

\medskip
Le point iii) résulte du fait que $b_\delta$ est
un monôme en $p_i$ (a fortiori ne dépend pas de $q_i$) et d'autre
part que $\omegaRes{\uQ}(X^\alpha)$ est homogène de poids $\widehat {d_i}-1$ en les coefficients $p_i,q_i$ de $Q_i$.

\medskip

Quant au point iv), c'est une redite du point ii) du résultat~\ref{qContributionTheorem}.
\end{proof}

\bigskip

Nous avons qualifiée la formule de l'énoncé formule de Macaulay car il
s'agit effectivement d'un cas particulier de la formule générale:
$$
\omega_\uP =  b_\delta(\uP)\,\omegaRes{\uP} \qquad
b_\delta(\uP) = \det W_{2,\delta}(\uP) 
$$
Une suite naturelle à ce théorème est la détermination précise du
résultant de $\uQ$ qui est, par définition, le scalaire défini par
$\Res(\uQ) := \omegaRes{\uQ}(\nabla_\uQ)$ (cf. la
définition~\ref{DefResultant} du résultant dans le chapitre à venir).
Nous avons établi une formule pour ce résultant : c'est la
formule~\ref{ResQconj} ; elle est \emph{conjecturale} dans le cas
général, mais pas dans le cas du jeu circulaire, ni du jeu
$\uR^{(1)}$, ni de nombreux autres jeux; bien entendu, cette formule,
sous sa forme conjecturale, ne sera pas utilisée dans la suite.

\medskip


\subsubsection {Le cas du jeu circulaire généralisé, où $\uR = (X_1^{d_1-1}X_2, X_2^{d_2-1}X_3,\dots, X_n^{d_n-1}X_1)$}

On rappelle (cf. chapitre~\ref{ChapJeuCirculaire}) que le jeu circulaire généralisé
$\uQ = \pXD + \bsq\,\uR$ est défini par
$$
Q_i \ = \ X_i^{d_i - 1}(p_iX_i + q_iX_{i+1}) 
\ =\ p_i X_i^{d_i} + q_i X_i^{d_i-1}X_{i+1}
$$ 
Allons plus loin en déterminant $\Res(\uQ)$.
Pour obtenir un déterminant bezoutien $\nabla_\uQ$, considérons
comme matrice bezoutienne de $\uQ$ la plus naturelle qui soit: une diagonale
principale de $p_iX_i^{d_i-1}$, une sous-diagonale de $q_i X_i^{d_i -
  1}$ et un petit bout en haut à droite égal à $q_n X_n^{d_n-1}$.
Par exemple pour $n=4$, on prend
$$
\dsV \ = \ 
\begin{bmatrix}
p_1X_1^{d_1-1}  &.             &.               &q_4X_4^{d_4-1} \\
\noalign {\smallskip}
q_1X_1^{d_1-1} &p_2X_2^{d_2-1}   &.            &. \\
\noalign {\smallskip}
.             &q_2X_2^{d_2-1}  &p_3X_3^{d_3-1} &. \\
\noalign {\smallskip}
.             &.                &q_3X_3^{d_3-1} &p_4X_4^{d_4-1} \\
\end{bmatrix}
$$
Or il est facile de calculer un déterminant d'une matrice de ce type ; en effet, on a :
$$
\Delta_4(\ua, \ub) =
\begin{vmatrix}
a_1  &.             &.               &b_4 \\
\noalign {\smallskip}
b_1 & a_2   &.            &. \\
\noalign {\smallskip}
.             &b_2  &a_3 &. \\
\noalign {\smallskip}
.             &.                &b_3 &a_4 \\
\end{vmatrix}
= a_1a_2a_3a_4 - b_1b_2b_3b_4 
$$
et de manière générale, pour un déterminant d'ordre $n$ de cette nature, on a :
$$
\Delta_n(\ua,\ub) \ = \ 
a_1 \cdots a_n - (-1)^n b_1 \cdots b_n \ =\ 
\prod_{i=1}^n a_i - \prod_{i=1}^n (-b_i)
$$
On obtient donc:
$$
\nabla_\uQ \ \overset{\rm def}{=} \ 
\det \dsV \ = \ 
\Big( \prod_{i=1}^n p_i \, - \,  \prod_{i=1}^n \overline{q}_i \Big) 
X^{\emouton}
$$
Ce qui donne pour $\Res(\uQ) = \omegaRes{\uQ}(\nabla_\uQ)$:
$$
\Res(\uQ) \ = \ 
\Big( \prod_{i=1}^n p_i \, - \, \prod_{i=1}^n \overline{q}_i \Big) 
\omegaRes{\uQ}(X^\emouton) =
\Big( \prod_{i=1}^n p_i \, - \, \prod_{i=1}^n \overline{q}_i \Big) 
\prod_{i=1}^n p_i^{\widehat d_i - 1}
\ = \ 
\prod_{i=1}^n p_i^{\widehat d_i}
- 
\prod_{i=1}^n \overline {q}_i p_i^{\widehat d_i-1}
$$
En particulier, si $p_i = 1 = \overline{q}_i$, alors $\Res(\uQ) = 0$, ce qui était attendu 
car, dans ce cas, $\xi = (1 : \cdots : 1)$ est un zéro commun aux $Q_i$.

\subsubsection {Le cas du jeu $\uR^{(1)} = (X_1^{d_1-1}X_2, X_1^{d_2-1}X_3, \dots, X_1^{d_n-1}X_1)$}

\index{jeu!monomial!de format $D$@$R^{(1)}$}%

Ce jeu figure après la proposition~\ref{ExistenceSuperHauteur} parmi les
exemples de jeux possédant une fonction hauteur.
Il est facile de déterminer une matrice bezoutienne de $\uQ := \pXD + \bsq\,\uR^{(1)}$ analogue
à celle utilisée pour le jeu circulaire généralisé, la différence résidant
dans le fait qu'en dehors de la diagonale, il y a des $q_i X_1^\sbullet$. Par exemple pour $n=4$:
$$
\dsV \ = \ 
\begin{bmatrix}
p_1X_1^{d_1-1}  &.             &.               &q_4X_1^{d_4-1} \\
\noalign {\smallskip}
q_1X_1^{d_1-1} &p_2X_2^{d_2-1}   &.            &. \\
\noalign {\smallskip}
.             &q_2X_1^{d_2-1}  &p_3X_3^{d_3-1} &. \\
\noalign {\smallskip}
.             &.                &q_3X_1^{d_3-1} &p_4X_4^{d_4-1} \\
\end{bmatrix}
$$
Ce qui conduit au déterminant bezoutien (en continuant à noter $\overline {q}_i = -q_i$):
$$
\nabla_\uQ \ = \ \det(\dsV) \ = \ p_1\cdots p_n\,X^\emouton - \overline{q}_1\cdots\overline{q}_n\,X_1^\delta
$$
Pour déterminer $\omegaRes{\uQ}(\nabla_Q)$, il suffit de connaître $\omegaRes{\uQ}(X_1^\delta)$, puisque l'on sait que 
$\omegaRes{\uQ}(X^\emouton) = p_1^{\widehat{d_1}-1} \cdots  p_n^{\widehat{d_n}-1}$.

\medskip
Dans la section~\ref{SectionJeuxSimples}, à un jeu monomial
$\uR = (R_1, \dots, R_n)$ de format $D$, nous avons attaché
la matrice $\Jac_\Un(\uX^D-\uR)$, matrice de $\bbM_n(\bbZ)$ dont les $n$ lignes
$\ell_1, \dots, \ell_n \in \bbZ^n$ (de somme nulle) sont données par
$$
\ell_i = \exposant(X_i^{d_i}) - \exposant(R_i)
$$
ainsi que ses $n$ mineurs principaux d'ordre $n-1$, notés dans l'ordre $r_1, \dots, r_n$.
Ceux-ci vérifient $\sum_i r_i \ell_i = 0$ ce qui conduit par sommation à l'identité:
$$
\exposant(R_1 \cdots R_n) - \exposant(X^D) = -\sum_{i=1}^n \ell_i =
\sum_{i=1}^n (r_i - 1) \ell_i
$$
Dans le cas du jeu $\uR^{(1)}$, on a $R_1\cdots R_n = X_1^\delta\, X_1\cdots X_n$
et l'identité ci-dessus peut s'écrire:
$$
\exposant(X_1^\delta) - \exposant(X^\emouton) = \sum_{i=1}^n (r_i - 1) \ell_i
$$
Enfin, pour ce jeu, on a déterminé les $n$ mineurs principaux~$r_i$ et
montré que $r_i \geqslant 1$ pour chaque $i$.
En conséquence, d'après le point iii) du théorème \ref{qContributionTheorem},
l'écriture ci-dessus est l'écriture positive minimale de
$\exposant(X_1^\delta) - \exposant(X^\emouton)$ sur les $\ell_i$.
Le théorème \ref{omegaresJeuSimple} fournit alors:

\begin {coro} \label{omegaresR1X1delta}
Pour $\uQ := \pXD + \bsq\,\uR^{(1)}$:
$$
\omegaRes{\uQ}(X_1^\delta) \ = \ \prod_{i=1}^n \overline {q}_i^{r_i-1}\, p_i^{\widehat d_i-r_i}
$$
\end {coro}

\medskip

\noindent
Commentaire. Il y a une propriété remarquable spécifique à ce jeu $\uR^{(1)}$.
On a vu (confer ce qui suit~\ref{ExistenceSuperHauteur}) que les $r_i$ sont donnés par:
$$
r_i = \prod_{j > i} d_j   \qquad \text{en particulier} \qquad
r_1 = \widehat {d_1}, \qquad  r_n = 1
$$
Ainsi, en examinant les exposants de $p_1$ et $\overline {q}_n$, 
on en déduit que $\omegaRes{\uQ}(X_1^\delta)$ ne dépend ni de
$p_1$, ni de $q_n$.

\medskip

En utilisant le déterminant bezoutien $\nabla_\uQ$ que nous avons
exhibé (combinaison de $X^\emouton$ et $X_1^\delta$) et le corollaire
précédent, on en déduit l'égalité ci-dessous, de même nature
que celle du jeu circulaire généralisé pour lequel chaque $r_i$ vaut 1.

\begin {coro} Le résultant du jeu $\uQ = \pXD +\bsq\,\uR^{(1)}$ est donné par
$$  
\Res(\uQ)\ = \ 
\prod_{i=1}^n p_i^{\widehat d_i} - \prod_{i=1}^n \overline {q}_i^{r_i} p_i^{\widehat d_i-r_i}
$$  
\end {coro}

\subsubsection {La formule (conjecturale) du résultant lorsque $\uR$ admet une fonction hauteur}

Voici la formule que l'on peut voir comme un mixte de combinatoire monomiale en degré $\delta$
et d'algèbre linéaire.

\begin {conj} \label {ResQconj}
Lorsque $\uR$ admet une fonction hauteur (sous-entendu sur $\bbN^n_\delta$), le résultant du
jeu $\uQ = \pXD +\bsq\,\uR$ est donné par
$$
\Res(\uQ) \ = \ 
\prod_{i=1}^n p_i^{\widehat d_i} -  \prod_{i=1}^n \overline {q}_i^{r_i} p_i^{\widehat d_i-r_i}
$$
\end {conj}

\medskip

Bien entendu, le fait que $\Res(\uQ)$ soit un tel binôme en $p_1,
\dots, p_n, q_1, \dots, q_n$ est fortement lié au statut \og posséder une fonction hauteur\fg{}
du jeu~$\uR$.
Voici un exemple où le résultant n'est pas de ce type. 
Considérons le format $D = (1,1,2)$ et le jeu $\uR$ suivant, 
où $X^\alpha = X_1^{\alpha_1} X_2^{\alpha_2} X_3^{\alpha_3}$
désigne n'importe quel monôme de degré 2:
$$
\uR = (X_2,\ X_1,\ X^\alpha),  \qquad\quad
Q_1 = p_1X_1 + q_1X_2, \quad
Q_2 = q_2X_1 + p_2X_2, \quad
Q_3 = p_3X_3^2 + q_3X^\alpha
$$
Alors $\Res(\uQ)$ est donné par la formule du premier cas d'école 
(cf. \ref{Resultant11d}), 
à savoir $\Res(\uQ) = Q_3(\xi_1, \xi_2,\xi_3)$ avec
$\xi_i = \det(Q_1, Q_2, X_i)$. Ici on a $\xi_1 = \xi_2 = 0$ et $\xi_3 = p_1p_2 - q_1q_2$ de sorte
que:
$$
\Res(\uQ) = p_3\,(p_1p_2-q_1q_2)^2 + q_3\,0^{\alpha_1} 0^{\alpha_2} (p_1p_2-q_1q_2)^{\alpha_3} 
$$
Par exemple, pour $X^\alpha = X_1^2$, on a $\Res(\uQ) = p_3\,(p_1p_2-q_1q_2)^2$ qui n'est pas un binôme.


\subsubsection*{A propos de l'énoncé conjectural \ref{ResQconj}} 

Nous allons prouver cet énoncé mais sous une hypothèse plus forte
que celle de~\ref{ResQconj}: celle où le jeu $\uR$ admet une fonction
hauteur sur $\bbZ^n_\delta$ (et pas seulement sur $\bbN^n_\delta$), ce
qui équivaut, en désignant par $(r_1, \dots, r_n)$ les $n$ mineurs
principaux de la matrice $A=\Jac_\Un(\uX^D-\uR)$ attachée à $\uR$, au fait que
$\pgcd(r_1, \dots, r_n) = 1$ et $r_i \geqslant 1$ pour tout $i$, cf la
proposition~\ref{ExistenceSuperHauteur}.

\subsubsection*{Premier maillon: un nouveau monôme $X^\gamma$ de degré $\delta$
sous la seule hypothèse $\forall i\ r_i \geqslant 1$}

On va montrer, c'est l'objet du lemme qui vient, que
$X_1X_2\cdots X_n \mid R_1R_2 \cdots R_n$, ce qui fait apparaître
un nouveau monôme de degré $\delta$
$$
X^\gamma = \frac{R_1R_2 \cdots R_n}{X_1X_2\cdots X_n}
$$
Par exemple, pour le jeu circulaire, on a $X^\gamma = X^\emouton$, tandis
que pour le jeu $\uR^{(1)}$, on a $X^\gamma = X_1^\delta$. Le lecteur peut
d'ailleurs consulter ce qui a été réalisé pour le jeu $\uR^{(1)}$, avant
le corollaire~\ref{omegaresR1X1delta}, car ce qui suit s'en inspire.

\medskip

En notant comme auparavant par $\ell_i$ les lignes de $A$, montrons l'égalité
suivante (liée aux chemins de $X^\gamma$ à $X^\emouton$ dans le graphe associé
à $\uR$):
$$
\gamma - (\emouton) = \sum_i (r_i-1) \ell_i
\leqno (\star)
$$
En effet, par définition de $\ell_i$, on a $\exposant(R_i) = \exposant(X_i^{d_i}) - \ell_i$
donc en sommant
$$
\exposant(R_1\cdots R_n) = \exposant(X^D) - \sum_i \ell_i
$$
d'où l'on tire
$$
\begin {array} {ccl}
  \gamma = \exposant(X^\gamma) &=& \exposant(R_1\cdots R_n) - \exposant(X_1\cdots X_n)
\\[2mm]
                             &=& \exposant(X^D) - \sum_i \ell_i - \exposant(X_1\cdots X_n)
\\[2mm]
                             &=& (\emouton) - \sum_i \ell_i 
\\
\end {array}
$$
En se souvenant de $\sum_i r_i\ell_i =0$ (cf lemme~\ref{riliSumRelation}), on en déduit $(\star)$.

Pour $\uQ = \bsp\,\uX^D + \bsq\,\uR$, l'égalité $(\star)$ et le théorème \ref{omegaresJeuSimple} fournissent:
$$
\omegaRes{\uQ}(X^\gamma) \ = \ \prod_{i=1}^n \overline {q}_i^{r_i-1}\, p_i^{\widehat d_i-r_i}
$$
ce qui constitue un premier maillon dans l'histoire. Avant d'aborder le deuxième maillon,
le lemme promis dans lequel il est inutile de supposer $\pgcd(r_1, \dots, r_n) = 1$.

\begin {lem} \label{X1X2XndiviseR1R2Rn}\leavevmode

Sous le couvert de la seule hypothèse $r_j \geqslant 1$ pour un $j$ donné, on a $X_j \mid \prod_{i\ne j} R_i$, à fortiori
$X_j \mid \prod_i R_i$.   En conséquence, si tous les $r_j$ sont $\geqslant 1$, on a
$X_1\cdots X_n \mid R_1\cdots R_n$. 
\end {lem}  

\begin {proof} \leavevmode

Pour simplifier, supposons $j=1$ et montrons le résultat sous la forme $X_1 \nmid R_2\cdots R_n
\Rightarrow r_1= 0$. Réécrivons la définition de la première colonne de
$A$ à l'aide des exposants des $R_i = X_1^{\alpha_{i,1}} X_2^\sbullet \cdots X_n^\sbullet$:
$$  
  A = \begin {bmatrix}
    *        &\cdots  & *\\
-\alpha_{2,1} &   \\
\vdots       &   B \\
-\alpha_{n,1} &\cdots &* \\
\end {bmatrix}
$$
Dans ce dessin, $B$ est la sous-matrice $(n-1)\times (n-1)$ de $A$ et on a $r_1 = \det(B)$. L'hypothèse
$X_1 \nmid R_2\cdots R_n$ se traduit par $\alpha_{i,1} = 0$ pour $i \geqslant 2$. Mais la somme des colonnes
de $A$ est nulle, donc la somme des colonnes de $B$ est également nulle d'où $\det(B) = 0$.
\end {proof}

\subsubsection*{Second maillon: sus aux déterminants bezoutiens de certains systèmes monomiaux $\uR$}

On \emph {voudrait bien disposer}, pour $\uQ = \bsp\,\uX^D +
\bsq\,\uR$, d'un déterminant bezoutien $\nabla_\uQ$ de $\uQ$ tel que
$$
\nabla_\uQ = p_1\cdots p_n X^\emouton -  \overline{q}_1\cdots\overline{q}_n X^\gamma
$$
Supposons ce souhait réalisé. En appliquant $\omegaRes{\uQ}$:
$$
\begin {array} {ccl}
\Res(\uQ) &=& p_1\cdots p_n\,\omegaRes{\uQ}(X^\emouton) -
    \overline{q}_1\cdots\overline{q}_n\,\omegaRes{\uQ}(X^\gamma)
\\[2mm]
&=& p_1\cdots p_n\,\prod_{i=1}^n p_i^{\widehat d_i-1}  -
\overline{q}_1\cdots\overline{q}_n\, \prod_{i=1}^n \overline {q}_i^{r_i-1}\, p_i^{\widehat d_i-r_i}
\\[2mm]
&=& \prod_{i=1}^n p_i^{\widehat d_i}  - \prod_{i=1}^n \overline {q}_i^{r_i}\, p_i^{\widehat d_i-r_i}
\\
\end {array}
$$
C'est exactement la formule de l'énoncé conjectural~\ref{ResQconj}.

Mais les auteurs ignorent s'il existe un tel déterminant bezoutien. 
Ils ont cependant constaté, de manière expérimentale, sur de nombreux exemples, que pour
\emph {n'importe quel} déterminant bezoutien $\nabla_\uQ$, on a seulement une congruence:
$$
\nabla_\uQ  \equiv p_1\cdots p_n X^\emouton -  \overline{q}_1\cdots\overline{q}_n X^\gamma
\bmod \langle\uQ\rangle
$$
Ce qui permet, par le même raisonnement, en appliquant $\omegaRes{\uQ}$, de conclure
à la formule du résultant. 

\bigskip

Cette congruence fait l'objet du théorème suivant.  Mais attention,
celui-ci fait apparaître un \emph {nouveau} contexte:
celui d'un jeu monomial $\uR$ vérifiant
seulement $r_i \geqslant 1$ pour tout $i$ où $(r_1, \dots, r_n)$ sont les
$n$ mineurs principaux de la matrice attachée à~$\uR$, sans
imposer pour autant $\pgcd(r_1, \dots, r_n) = 1$.  Ce contexte
suffit pour assurer $X_1\cdots X_n \mid R_1\cdots R_n$,
(cf. lemme~\ref {X1X2XndiviseR1R2Rn}) donnant naissance au monôme de
degré $\delta$ que nous avons noté~$X^\gamma$:

$$
X^\gamma = \frac{R_1R_2 \cdots R_n}{X_1X_2\cdots X_n}
$$

\begin {theo} [Déterminant bezoutien d'un système monomial lorsque $r_i\geqslant 1,\ \forall i$]
\label{MonomialNabla}

Dans ce contexte,  tout déterminant bezoutien $\nabla_\uQ$ de
$\uQ = \bsp\,\uX^D + \bsq\,\uR$ vérifie  
$$
\nabla_\uQ  \equiv p_1\cdots p_n X^\emouton -  \overline{q}_1\cdots\overline{q}_n X^\gamma
\bmod \langle\uQ\rangle
$$
\end {theo}

Commentaire. Il y a donc 2 contextes proches l'un de l'autre mais
distincts. Il y a celui d'un jeu admettant une fonction hauteur sur
$\bbN^n_\delta$ pour lequel on propose une formule pour le
résultant~(\ref{ResQconj}). Et il y a celui ci-dessus concernant le
déterminant bezoutien. Ce dernier permet la résolution du premier avec
cependant une hypothèse supplémentaire: on demande l'existence d'une
fonction hauteur sur $\bbZ^n_\delta$ pour~$\uR$.

A noter que dans le contexte ci-dessus, $\Res(\uQ)$ n'a aucune raison
d'être un binôme en les $p_i, \overline {q}_j$.
Idem, aucune raison que $\omegaRes{\uQ}(X^\alpha)$ soit un monôme, même pour
$X^\alpha = X^\emouton$.

\medskip

La preuve figure en page~\pageref{MonomialNablaProof}, mais nous devenons
au préalable étudier le calcul du déterminant de certaines matrices
bezoutiennes.

\subsubsection*{Matrice bezoutienne d'un système monomial}

Revenons sur la construction d'une matrice bezoutienne $\dsV_\uR$ d'un
système \emph{monomial} $\uR$ de format~$D$, ce qui fournira une matrice
bezoutienne pour $\uQ$. Voici une stratégie élémentaire pour laquelle
dans chaque colonne, il y a un seul coefficient non nul. Rappelons
qu'il s'agit de réaliser $\begin{bmatrix} R_1& \cdots& R_n\end{bmatrix} =
\begin{bmatrix} X_1 & \cdots & X_n\end{bmatrix}\,\dsV_\uR$: pour chaque $R_j$, on sélectionne un
$X_i$ tel que $X_i \mid R_j$ et on place $R_j/X_i$ en colonne~$j$,
ligne $i$. Par exemple, pour $n=3$:
$$
[R_1, R_2, R_3] = [X_1, X_2, X_3] \
\begin{bmatrix}
.               &.                 &. \\
\frac{R_1}{X_2} &.                 &\frac{R_3}{X_2} \\
.               &\frac{R_2}{X_3}   &. \\
\end {bmatrix}
$$
\label{NOTA10-MonomialBezMat}%
Ceci suppose bien entendu que $X_2 \mid R_1$, $X_3 \mid R_2$ et $X_2 \mid R_3$. Nous
en considérerons un exemple plus loin et pour l'instant, nous nous appuyons sur ce cas
de figure. Qu'en est-il de $\dsV_\uQ$?
$$
\dsV_\uQ =
\begin {bmatrix}
p_1X_1^{d_1-1} &. & . \\
.             &p_2X_2^{d_2-1} &. \\
.             &.             &p_3X_3^{d_3-1} \\
\end {bmatrix}
+
\begin {bmatrix}
.                 &.                 &. \\
q_1\frac{R_1}{X_2} &.                 &q_3\frac{R_3}{X_2} \\
.                 &q_2\frac{R_2}{X_3} &. \\
\end {bmatrix}
$$
\label{VQelementary}%
On voit donc apparaître des matrices d'une certaine nature, typiquement pour $n=4$:
$$
\begin {bmatrix}
a_1 &.   &b_3 &.  \\
b_1 &a_2 &.   &b_4 \\
.   &.   &a_3 &.   \\
.   &b_2 &.   &a_4 \\
\end {bmatrix}
$$
dont le déterminant se calcule via:
$$
(a_1\varepsilon_1 + b_1\varepsilon_2) \wedge
(a_2\varepsilon_2 + b_2\varepsilon_4) \wedge
(a_3\varepsilon_3 + b_3\varepsilon_1) \wedge
(a_4\varepsilon_4 + b_4\varepsilon_2) 
$$
De manière générale, si $b_1$ est en ligne $\ell_1$, $b_2$ en
ligne $\ell_2$ etc. le déterminant $\Delta$ qui apparaît est donné par
$$
\Delta\, (\varepsilon_1 \wedge \cdots \wedge\varepsilon_n) =
(a_1\varepsilon_1 + b_1\varepsilon_{\ell_1}) \wedge
(a_2\varepsilon_2 + b_2\varepsilon_{\ell_2}) \wedge \cdots \wedge
(a_n\varepsilon_n + b_n\varepsilon_{\ell_n}) 
$$
Nous aurons besoin de la formule banale ci-dessous que nous fournissons
sous forme d'un lemme sans preuve.

\begin {lem} \label{DeltaFormula}

Le développement du membre droit fournit une somme portant sur les parties $J$
de~$\{1..n\}$ dont on note $\overline J$ le complémentaire:
$$
\Delta = \sum_J s_J\, \prod_{i \in \overline J} a_i \prod_{j \in J} b_j 
$$
dans laquelle $s_J \in \{0, \pm 1\}$ est un signe étendu.

\medskip
Ce signe étendu se détermine comme suit, à l'aide de la fonction $\tau_J$:
$$
\tau_J : \{1..n\} \to \{1..n\}  \qquad
\tau_J(j) = \begin {cases}
\ell_j &\text{si } j\in J   \\
j     &\text{sinon}
\end {cases}
$$
Alors $s_J = 0$ si $\tau_J$ n'est pas une permutation, sinon $s_J$ est
la signature de $\tau_J$. En utilisant une notation fonctionnelle pour
$\ell$, i.e. $\ell : \{1..n\} \to \{1..n\}$, donc $\ell(j)$ au lieu de
$\ell_j$, alors $\tau_J$ est une permutation si et seulement si
$\ell(J) = J$.  Dans ce cas $\tau_J(J) = J$ et $\tau_J$ est l'identité
sur $\overline J$, si bien que $s_J$ est aussi la signature de la
permutation ${\tau_J}_{|J} = \ell_{|J}$ (l'indice $_{|J}$ désigne la
restriction à $J$).
\end {lem}

\medskip

La preuve  du théorème \ref{MonomialNabla} va nécessiter un certain nombre
de lemmes de nature combinatoire. Nous préférons commencer par un
petit exemple.

\subsubsection*{Un exemple élémentaire}

Le déterminant de $\dsV_\uQ$ donné à la page~\pageref{VQelementary} est facile à calculer, soit
de manière directe, soit via l'algèbre extérieure, car du type:
$$
\begin {vmatrix}
a_1 &.   &. \\
b_1 &a_2 &b_3 \\
.   &b_2 &a_3 \\
\end {vmatrix}
\qquad
(a_1\varepsilon_1 + b_1\varepsilon_2) \wedge
(a_2\varepsilon_2 + b_2\varepsilon_3) \wedge
(a_3\varepsilon_3 + b_3\varepsilon_2) =
(a_1a_2a_3 - a_1b_2b_3)\, \varepsilon_1 \wedge \varepsilon_2  \wedge \varepsilon_3
$$
Ce qui donne
$$
\nabla_\uQ = p_1p_2p_3X^\emouton - p_1q_2q_3X_1^{d_1-1}\frac{R_2R_3}{X_2X_3}
$$
Ce n'est \emph {pas} le binôme convoité:
$$
p_1p_2p_3X^\emouton - \overline{q}_1 \overline{q}_2 \overline{q}_3 X^\gamma \overset{\rm def.}{=}
p_1p_2p_3X^\emouton - \overline{q}_1 \overline{q}_2 \overline{q}_3 \frac{R_1R_2R_3}{X_1X_2X_3}
$$
Cependant, modulo $\langle\uQ\rangle$, on dispose, si $X_1^{d_1}\mid X^\alpha$, de la règle de réécriture suivante:
$$
p_1 X^\alpha \equiv \overline{q}_1 \frac{X^\alpha}{X_1^{d_1}}R_1   \bmod \langle\uQ\rangle
$$
Peut-on l'appliquer à $X^\alpha := X_1^{d_1-1}\frac{R_2R_3}{X_2X_3}$? C'est le cas si et
seulement si $X_1$ divise $\frac{R_2R_3}{X_2X_3}$
i.e. si $X_1X_2X_3$ divise $R_2R_3$, auquel cas, par définition de $X^\alpha$:
$$
\frac{X^\alpha}{X_1^{d_1}}\,R_1 = \frac{R_1R_2R_3}{X_1X_2X_3} = X^\gamma
$$
Plaçons nous dans ce cas $X_1X_2X_3$ divise $R_2R_3$. Alors
$$
\nabla_\uQ = p_1p_2p_3X^\emouton - p_1q_2q_3X_1^{d_1-1}\frac{R_2R_3}{X_2X_3} \equiv
p_1p_2p_3X^\emouton - \overline{q}_1 \overline{q}_2 \overline{q}_3 X^\gamma \quad
\bmod \langle\uQ\rangle
$$
Youpi!

\bigskip

Voici un jeu tel $\uR$, de format $D = (d_1, d_2,d_3)$, la matrice $A$ attachée, de lignes
$\ell_i = \exposant(X_i^{d_i}) - \exposant(R_i)$ et les 3 mineurs principaux $r_i$ de $A$.
$$
\uR = (X_2^{d_1}, X_1^{d_2-1}X_3, X_3^{d_3-1}X_2),
\qquad\qquad
A = \begin{bmatrix}
d_1     &-d_1   &0 \\
1-d_2 &d_2  &-1 \\
  0   &-1   &1 \\
\end{bmatrix},
\qquad
(r_1, r_2, r_3) = (d_2-1, d_1, d_1)
$$
On suppose $d_2 \geqslant 2$ de manière à avoir $r_1 \geqslant 1$. On a bien
les conditions requises:
$$
X_2 \mid R_1, \quad X_3 \mid R_2, \quad X_2\mid R_3, \qquad X_1X_2X_3 \mid R_2R_3
$$
la dernière divisibilité étant due à $d_2 \geqslant 2$.

\medskip

Disposait-on d'autres choix pour construire $\dsV_\uR$ toujours avec la même stratégie:
$R_j/X_i$ en colonne~$j$, ligne~$i$, lorsque $X_i \mid R_j$? 
Pour $R_1$, on n'a pas le choix, il faut prendre $i=2$. Pour $R_2$, on peut prendre $i=1$
(car $d_2 \geqslant 2$) ou $i=3$. Quant à $R_3$, $i=2$ convient toujours
et si $d_3 \geqslant 2$, on peut prendre $i=3$. D'où d'autres possibilités:
$$
\dsV_\uR : \qquad
\begin{bmatrix}
.               &\frac{R_2}{X_1}  &. \\
\frac{R_1}{X_2} &.                &\frac{R_3}{X_2}\\
.               &.               &. \\
\end {bmatrix}
\qquad\qquad
\begin{bmatrix}
.               &.                 &. \\
\frac{R_1}{X_2} &.                 &. \\
.               &\frac{R_2}{X_3}   &\frac{R_3}{X_3} \\
\end {bmatrix}
\qquad\qquad
\begin{bmatrix}
.               &\frac{R_2}{X_1}   &. \\
\frac{R_1}{X_2} &.                 &. \\
.               &.                 &\frac{R_3}{X_3} \\
\end {bmatrix}
$$
conduisant à des matrices bezoutiennes pour $\uQ = \bsp\,\uX^D + \bsq\,\uR$
obtenues comme somme de la matrice diagonale $\diag(p_1X_1^{d_1-1}, p_2X_2^{d_2-1}, p_3X_3^{d_3-1})$
et de $\dsV_{\bsq\uR}$ moulée sur $\dsV_{\uR}$.
Aucune n'a un déterminant égal au binôme convoité 
$$
p_1p_2p_3X^\emouton - \overline{q}_1 \overline{q}_2 \overline{q}_3 X^\gamma
$$
Mais  tous les $\nabla_\uQ$ sont égaux modulo $\langle\uQ\rangle$.
Et comme on en a exhibé un égal modulo $\langle\uQ\rangle$
au binôme convoité, il en est de même de n'importe quel $\nabla_\uQ$.

\medskip

Dans cet exemple où $(r_1, r_2, r_3) = (d_2-1, d_1, d_1)$, imposons,
en plus de $r_1 \geqslant 1$, la condition arithmétique $\pgcd(r_1,r_2,r_3) = 1$
i.e. $(d_2-1) \wedge d_1 = 1$. Alors $\uR$ possède une fonction
hauteur sur $\bbZ^3_\delta$ et en conséquence:
$$
\Res(\uQ) = p_1^{\widehat d_1} p_2^{\widehat d_2} p_3^{\widehat d_3}
- \overline q_1^{r_1} p_1^{\widehat d_1-r_1} \overline q_2^{r_2} p_2^{\widehat d_2-r_2}
\overline q_3^{r_3} p_3^{\widehat d_3-r_3}
$$
Cette formule n'est \emph{pas} valide sans la condition
arithmétique $(d_2-1) \wedge d_1 = 1$, le résultant est
beaucoup plus compliqué et n'est pas un binôme en les $p_i,\overline q_j$.

\subsubsection {Trois lemmes combinatoires}

Pour une partie $J \subset \{1..n\}$, nous adoptons les notations:
$$
X_J = \prod_{j \in J} X_j, \qquad\qquad
R_J = \prod_{j \in J} R_j
$$
Le lemme qui suit est une généralisation du lemme~\ref{X1X2XndiviseR1R2Rn}

\label{NOTA10-produitXJRJ}%
%
%

\begin {lem}\label {XiDivRJlemma} \leavevmode 

\begin {enumerate} [\rm i)]
\item
Soit $\uR$ un jeu monomial quelconque.
Pour tout $J \subset \{1..n\}$ \emph {non vide} de complémentaire $\overline J$, on a l'implication:
$$
[\forall i \in \overline J \quad  X_i \nmid R_J]  \quad\Longrightarrow\quad
[\forall i \in \overline J \quad r_i=0]
$$

\item
Soit $\uR$ un jeu monomial vérifiant $r_i \geqslant 1$ pour tout $i$. Alors pour toute
partie \emph {propre et non vide} $J$ de $\{1..n\}$, il existe $i \in \overline J$
tel que $X_i \mid R_J$.
\end {enumerate}
\end {lem}

\begin {proof} \leavevmode

i) Pour simplifier, supposons $J = \{4..n\}$. L'hypothèse est donc:
$$
X_i \nmid R_4 \cdots R_n   \qquad i=1,2,3
$$
Et il faut démontrer que $r_1=r_2=r_3 = 0$. Réécrivons la définition des 3 premières colonnes de
$A$ à l'aide des exposants des $R_k$:
$$
R_k = X_1^{\alpha_{k,1}} X_2^{\alpha_{k,2}} \cdots X_n^{\alpha_{k,n}}
$$
en notant, pour des raisons typographiques, $\overline\alpha_{i,j}$ pour $-\alpha_{i,j}$:

$$
\def\al{\overline\alpha}
A =
\begin {bmatrix}
    *      &*       &*        &*       &       &*\\
    *     &\sbullet    &\sbullet        &*       &       &\\
    *     &\sbullet        &\sbullet        \\
\al_{4,1} &\al_{4,2} &\al_{4,3} &*      &\cdots &*   \\
\vdots   &\vdots    &\vdots   &\vdots &B      &\vdots  \\
\al_{n,1} &\al_{n,2} &\al_{n,3} &*      &\cdots &* \\
\end {bmatrix}
=
\begin {bmatrix}
    *     &*       &*        &*       &       &*\\
    *    &\sbullet        &\sbullet        &*       &       &\\
    *    &\sbullet        &\sbullet        \\
    0    &0         &0        &*      &\cdots &*   \\
\vdots   &\vdots    &\vdots   &\vdots &B       &\vdots  \\
    0    &0         &0 &*      &\cdots &* \\
\end {bmatrix}
$$
A droite, la nullité des 3 sous-colonnes provient de l'hypothèse $X_i \nmid R_4\cdots R_n$ pour $i =1,2,3$,
équivalente à $\alpha_{\ell,1} =\alpha_{\ell,2} =\alpha_{\ell,3} = 0$ pour $\ell > 3$.
Montrons que $r_1=0$ par exemple en utilisant la sous-matrice $(n-3)\times(n-3)$ de $A$ notée $B$. Comme
la somme des colonnes de $A$ est nulle, il en est de même de $B$, donc $\det(B) = 0$. Et le déterminant $r_1$
de la sous-matrice $A_{\{2..n\} \times \{2..n\}}$ est le produit du mineur $2\times 2$ visible avec des
$\sbullet$ et de $\det(B)$, qui est nul. Bilan: $r_1=0$. Et de même, $r_2=r_3=0$. 

\medskip

ii) Comme $r_i \geqslant 1$ pour tout $i$, la seule partie non vide $J$ qui
vérifie $[\forall i \in \overline J \quad r_i=0]$ est $J = \{1..n\}$,
qui a été exclue par hypothèse. D'où la négation de
$[\forall i \in \overline J \quad X_i \nmid R_J]$, ce qui est
le résultat à montrer.
\end {proof}

\medskip

Commentaires. Dans le i), il est impératif d'exclure $J = \emptyset$ pour lequel
$\overline J = \{1..n\}$ et $R_\emptyset = 1$: l'hypothèse est donc toujours
vérifiée mais pas la conclusion. Dans ii), il est impératif d'exclure en plus
$J = \{1..n\}$, car $\overline J = \emptyset$  pour lequel il est hors de
question de démontrer qu'il existe $i \in \emptyset \text{ tel que ...}$.
Enfin, dans le point i), en prenant $J = \{1..n\} \setminus \{j\}$, donc $\overline J = \{j\}$
et en supposant $r_j \geqslant 1$, on obtient, par contraposée, $X_j \mid \prod_{i\ne j} R_i$, ce
qui constitue le lemme~\ref{X1X2XndiviseR1R2Rn}.

\bigskip

Nous allons maintenant étudier certains termes définis à partir d'un
jeu $\uR$ vérifiant $r_i \geqslant 1$ pout tout $i$ et d'une partie $J$ de
$\{1..n\}$. En une première approximation, il s'agit de termes qui
interviennent dans le développement du déterminant d'une matrice
bezoutienne de $\bsp\,\uX^D + \bsq\,\uR$ selon la formule du
lemme~\ref{DeltaFormula}, voir aussi,
page~\pageref{MonomialNablaProof} le début de la preuve du
théorème~\ref{MonomialNabla}.
En désignant par $I$ le complémentaire de~$J$, nous posons
$$
T_J = \prod_{i\in I} p_i X_i^{d_i-1} \times \prod_{j\in J}q_j \times \frac {R_J}{X_J}
\qquad
\framebox [1.05\width][c]{sous la condition de divisibilité $X_J \mid R_J$}
$$
Par exemple, $T_\emptyset$ est toujours défini, $T_{\{1..n\}}$ l'est grâce
au lemme \ref{X1X2XndiviseR1R2Rn} et
$$
T_\emptyset = p_1 \cdots p_n\, X^\emouton, \qquad\qquad
T_{\{1..n\}} = q_1 \cdots q_n\, X^\gamma, \qquad\qquad
X^\gamma = \frac{R_1\cdots R_n}{X_1 \cdots X_n}
$$

\begin {lem} \leavevmode \label{TJcongruence}

Soit $\uR$ un jeu monomial vérifiant $r_i \geqslant 1$ pour tout $i$.

\begin {enumerate} [\rm i)]
\item
Soit $J$ une partie \emph {propre et non vide} de $\{1..n\}$, de complémentaire $I$, 
tel que $T_J$ soit défini. Alors pour tout $i_0 \in I$ vérifiant 
$X_{i_0} \mid R_J$ (il en existe d'après le lemme~\ref{XiDivRJlemma}), le terme $T_{J \cup \{i_0\}}$
est défini et on a la congruence:
$$
T_{J\cup\{i_0\}} \equiv -T_J \bmod \langle\uQ\rangle
$$

\item
Pour toute partie  \emph{non vide} $J$ de $\{1..n\}$ telle que
$T_J$ soit  définie, on a:
$$
T_J \equiv (-1)^{n -\#J}\, q_1\cdots q_n\ X^\gamma  \bmod \langle\uQ\rangle
$$
\end {enumerate}
\end {lem}

\begin {proof} \leavevmode

i) Comme $X_J \mid R_J$, $X_{i_0} \mid R_J$ et $i_0 \notin J$, on a
$X_{J\cup\{i_0\}} \mid R_J$, a fortiori $X_{J\cup\{i_0\}} \mid
R_{J\cup\{i_0\}}$ donc $T_{J\cup\{i_0\}}$ est défini. En écrivant:
$$
\frac{R_J}{X_J} = X_{i_0} \times \frac {R_J}{X_{J\cup\{i_0\}}} 
\qquad \qquad
\prod_{i\in I} p_i X_i^{d_i-1} =  p_{i_0}\,X_{i_0}^{d_{i_0}-1}\times\prod_{i\in I \setminus \{i_0\}} p_i X_i^{d_i-1}
$$
On obtient la forme de $T_J$:
$$
T_J = p_{i_0} X_{i_0}^{d_{i_0}}  \times 
\prod_{i\in I \setminus \{i_0\}} p_i X_i^{d_i-1} \times \prod_{j\in J}q_j \times \frac {R_J}{X_{J\cup\{i_0\}}}
$$
En y remplaçant $p_{i_0} X_{i_0}^{d_{i_0}}$ par $-q_{i_0} R_{i_0}$ qui lui est congru modulo
$\langle\uQ\rangle$ (par définition de $\uQ$), le lecteur vérifiera que l'on
obtient $-T_{J\cup\{i_0\}}$.

\medskip

ii) Le résultat s'obtient par récurrence sur $n-\#J$. Si $n-\#J = 0$ i.e. 
$J = \{1..n\}$, cela découle de $T_{\{1..n\}} = q_1 \cdots q_n\, X^\gamma$. Si $\#J < n$,
on utilise le point précédent dans la récurrence.
\end {proof}

\medskip

Il est bien connu qu'étant donné un ensemble fini de cardinal $n$,
son nombre de parties est $2^n$. Et dans le cas où l'ensemble n'est
pas vide, le nombre de parties de cardinal pair est égal au nombre de
nombre de parties de cardinal impair. Ce qui résulte du
développement de Newton via les coefficients binomiaux où $n \geqslant 1$:
$$
0 = (1 - 1)^n = \sum_{i\ \rm pair} \binom{n}{i} - \sum_{j\ \rm impair} \binom{n}{j} 
$$
On peut considérer que le lemme qui vient utilise de manière cruciale
ce résultat tout en le généralisant (prendre $f = \id_X$).

\begin {lem} \label{PatriceLemma} \leavevmode

Soit $X$ un ensemble fini et $f : X \to X$ une application. On note $\Gamma$ l'ensemble
des $Y \subseteq X$ tels que $f(Y) = Y$ ou de manière équivalente $Y \subseteq f(Y)$.
Ainsi $\emptyset \in \Gamma$. Pour $Y \in \Gamma$, on note $f_Y$ la restriction de $f$ à~$Y$:
c'est une permutation de $Y$.

\begin {enumerate} [\rm i)]
\item
L'ensemble $\Gamma$ est stable par réunion. En conséquence, il possède
un plus grand élément $Y_0$ pour l'inclusion et $Y_0$ est l'unique
élément de $\Gamma$ de cardinal maximum.

\item
On suppose $X$ non vide.
On désigne, pour $Y \in \Gamma$, par $\rho(Y) = (-1)^{\#Y}\epsilon(Y)$
où $\epsilon(Y)$ la signature de la permutation $f_Y$, et on définit
$\Gamma^+$, $\Gamma^-$ par
$$
\Gamma^s = \{Y \in \Gamma \mid   \rho(Y) = s \}  \qquad  s = \pm 1
$$
Alors, $\#\Gamma^+ = \#\Gamma^-$ et $\#\Gamma$ est une puissance de 2
(une \og vraie\fg{} autre que 1).

De manière plus précise, $Y_0$ est non vide. Désignons par $X_i$ les orbites de
la permutation $f_{Y_0}$:
$$
Y_0 = X_1 \sqcup \cdots \sqcup X_k
$$
Alors $\Gamma$ est constitué des parties qui sont des réunions de ces
orbites, et $\Gamma^+$ (resp. $\Gamma^-$) est constitué des 
réunions d'un nombre pair (resp. impair) d'orbites. Ainsi
$\#\Gamma = 2^k$ et $\#\Gamma^+ = \#\Gamma^- = 2^{k-1}$.

\item
L'ensemble $\Gamma$ est un treillis pour $(\subseteq,\cup,\cap)$ avec un plus petit et plus grand élément,
stable par l'involution $Y \mapsto Y_0\setminus Y$.
\end {enumerate}
\end {lem}

\begin {proof} \leavevmode

i) Clair. Le problème $(X,f)$ se ramène à $(Y_0,f_{Y_0})$ qui est une permutation,
à condition de montrer que $Y_0$ est non vide si $X$ est non vide.

\medskip

ii) Montrons que $Y_0$ est non vide. Fixons un $x \in X$ ($X$ est non vide)
et considérons la suite $(f^m(x))_{m \in \bbN}$. Puisque $X$ est fini, il
existe $m_1 < m_2$ tels que $f^{m_1}(x) = f^{m_2}(x)$. Notons $Y$
l'ensemble \emph{non vide}:
$$
Y = \{f^m(x) \mid  m_1 \leqslant m < m_2\}
$$
Alors $f(Y) = Y$ puisque:
$$
f(Y) = \{f^{m+1}(x) \mid  m_1 \leqslant m < m_2\} =
\{f^{m'}(x) \mid  m_1 < m' \leqslant m_2\} = \{f^m(x) \mid  m_1 \leqslant m < m_2\}
$$
Comme $Y$ est non vide, il en est de même de $Y_0$ qui le contient.

La signature d'un cycle fournit  $\rho(X_i) = -1$ pour toute orbite $X_i$.
Le reste ne pose pas de difficulté une fois remarqué que $\rho$ est
un \og morphisme\fg{} au sens où, pour deux parties \emph {disjointes} $Y,
Y' \in \Gamma$:
$$
\rho(Y \sqcup Y') =\rho(Y)\rho(Y') 
$$

iii) Cela résulte de la description de $\Gamma$. On peut cependant donner
une preuve directe de la stabilité par complémentaire.
Il suffit de montrer que $Y_0 \setminus Y \subseteq f(Y_0 \setminus Y)$. Soit $y \in
Y_0 \setminus Y$; il existe $x \in Y_0$ tel que $y = f(x)$ et $x \notin Y$ puisque
$f(x) \notin f(Y) = Y$.
\end {proof}

\begin {proof} [Preuve du théorème \ref{MonomialNabla}]
\label{MonomialNablaProof}

Soit une matrice bezoutienne $\dsV_\uR$ de $\uR$ avec  un seul coefficient non nul par colonne,
$Q_j/X_{\ell(j)}$ en colonne $j$, en ligne $\ell(j)$, étant entendu que $X_{\ell(j)} \mid Q_j$.
On adopte une notation fonctionnelle $\ell : \{1..n\} \to \{1..n\}$.
Cette matrice $\dsV_\uR$ détermine une matrice bezoutienne $\dsV_\uQ$.
On utilise la formule du déterminant $\Delta$ donnée dans le lemme~\ref{DeltaFormula}. 
Elle se limite aux $J$ tels que $s_J \ne 0$ i.e. $\ell(J) = J$; 
comme $X_{\ell(l)} \mid R_j$, on a donc pour un tel $J$
$$
\prod_{j \in J} X_j = \prod_{j\in J} X_{\ell(j)}  \mid \prod_{j \in J} R_j
$$
I.e. avec nos notations antérieures $X_J \mid R_J$. Ainsi le terme $T_J$ est défini, on vérifie
que c'est bien lui qui intervient dans la formule $\Delta$ du lemme \ref{DeltaFormula}
via l'expression:
$$
\nabla_\uQ = \sum_J  s_J T_J  \qquad
\framebox [1.05\width][c]{la sommation ayant lieu sur les $J$ tels que $\ell(J) =J$}
$$
On met $J = \emptyset$ de côté pour lequel $s_J T_J = p_1\cdots p_nX^\emouton$. Pour les
autres $J \ne \emptyset$, on a d'après le lemme~\ref{TJcongruence}:
$$
s_J T_J \equiv s_J (-1)^{n-\#J}\, q_1 \cdots q_n\, X^\gamma  \bmod \langle\uQ\rangle
$$
si bien que
$$
\nabla_\uQ \equiv p_1\cdots p_nX^\emouton + (-1)^n\, s\, q_1 \cdots q_n\, X^\gamma  \bmod \langle\uQ\rangle
\qquad \text{avec} \quad
s = \sum_{J \ne \emptyset} s_J (-1)^{\#J}
$$
Utilisons le lemme~\ref{PatriceLemma}. Comme nous avons mis de côté
$J=\emptyset$ pour lequel $s_J (-1)^{\#J} = 1$, le nombre de $J$ qui
interviennent dans la somme $s$ ci-dessus est impair, disons $2k+1$.
Et il y a~$k$ parties $J$ pour lesquelles $s_J (-1)^{\#J} = 1$ et $k+1$ pour
lesquelles $s_J (-1)^{\#J} = -1$. Ce qui fait qu'il reste $s = -1$. 
En définitive:
$$
\begin {array} {ccl}
\nabla_\uQ &\equiv&
p_1\cdots p_nX^\emouton - (-1)^n\, q_1 \cdots q_n\, X^\gamma \bmod \langle\uQ\rangle
\\
&\equiv& p_1\cdots p_nX^\emouton -  \overline q_1 \cdots \overline q_n\, X^\gamma  \bmod \langle\uQ\rangle
\\
\end {array}
$$

\end {proof}


\subsection {Systèmes creux $\protect\uP$ tels que $\Gr(\protect\omegaresP) \geqslant 2$}
\label{SectSystemesCreuxGr2}

\index{système!creux}%
\index{système!réduit}%
\index{théorème!profondeur $\ge 2$ de la forme $\omegaRes{\uP}$ pour des systèmes $\uP$ spécifiques}%

L'objectif ici est d'exhiber des systèmes $\uPred$ ayant \emph {peu de
  coefficients}, ces coefficients étant des indéterminées, et tels que
$\Gr(\omegaresPred) \geqslant 2$.  A fortiori, d'après le théorème~\ref{GrViaPcouvreQ} 
(contrôle de la profondeur par la propriété de couverture), pour tout système~$\uP$
couvrant~$\uPred$, on a $\Gr(\omegaresP) \geqslant 2$. En voici un
exemple pour $n = 3$, le format $D = (d_1, d_2, d_3)$ étant
quelconque, en utilisant $9$ indéterminées:
$$
\uPred : \quad
\left\{
\begin {array} {*{7}c}
p_1X_1^{d_1} &+& q_1X_1^{d_1-1}X_2 &+&                   &+& s_1X_3^{d_1-1}X_2 \\
p_2X_2^{d_2} &+& q_2X_1^{d_2-1}X_3 &+& r_2X_2^{d_2-1} X_3                      \\
p_3X_3^{d_3} &+&                 &+& r_3X_2^{d_3-1}X_1 &+& s_3X_3^{d_3-1}X_1 \\
\end {array}
\right.
$$
\label{NOTA10-jeuPred}%
On est loin du nombre d'indéterminées d'un système générique qui, pour $n = 3$, vaut :
$$
\sum_{i=1}^n \binom {d_i+n-1}{n-1} =
\binom {d_1+2}{2} + \binom {d_2+2}{2} + \binom {d_3+2}{2} 
$$
Pour cela, nous allons utiliser le jeu $\uR^{(j)}$ suivant, analogue
au jeu $\uR^{(1)}$ déjà rencontré dans les pages précédentes, et qui,
en un certain sens, privilégie l'indéterminée $X_j$. Comme d'habitude
l'indexation est cyclique i.e. $X_{n+1} = X_1$:
$$
R^{(j)}_i = X_j^{d_i-1}X_{i+1} \quad  1 \leqslant i \leqslant n
\qquad\quad
\uR^{(j)} = (X_j^{d_1-1} X_2,\ X_j^{d_2-1} X_3,\ \dots\  ,
X_j^{d_{n-1}-1} X_n,\ X_j^{d_n-1} X_1)
$$
Les propriétés démontrées pour $\uR^{(1)}$ perdurent évidemment pour
le jeu $\uR^{(j)}$. En particulier, ce dernier
admet une fonction hauteur sur $\bbZ^n_0$ et on dispose d'une 
formule analogue à celle figurant dans le corollaire~\ref{omegaresR1X1delta}.
Nous la donnons de manière explicite mais le point le plus important c'est qu'elle ne
dépend ni de $p_j$ ni de $q_{j-1}$ (voir le commentaire après le
corollaire~\ref{omegaresR1X1delta}).

\label{NOTA10-jeuRj}%
%
%
\index{jeu!monomial!de format $D$@$R^{(j)}$}%

\begin {coro}
Pour $\uQ := \pXD + \bsq\,\uR^{(j)}$:
$$
\omegaRes{\uQ}(X_j^\delta) \ = \ \prod_{i=1}^n \overline {q}_i^{r_i-1}\, p_i^{\widehat d_i-r_i}
$$
où $(r_1, \dots, r_n)$ sont les $n$ mineurs principaux de la matrice $\Jac_\Un(\uX^D-\uR^{(j)})$
attachée au jeu~$\uR^{(j)}$. 
On a en particulier $r_{j-1} = 1$ et $r_j = \widehat {d}_j$.
En conséquence, l'évaluation à gauche ne dépend ni de $p_j$ ni de $q_{j-1}$ (en convenant
de la notation cyclique $q_0 = q_n$).
\end {coro}

\medskip

Nous pouvons maintenant expliquer la construction du système en nous allouant un
certain nombre d'indéterminées sur le modèle
$$
\bsp\uX^D + \bsq^{(1)}\uR^{(1)} + \cdots + \bsq^{(n)}\uR^{(n)}
$$
Pour éviter des notations trop lourdes, nous allons nous limiter à $n=3$ mais
ce qui suit est valide pour n'importe quel $n$. Auparavant,
il nous semble bon de redonner le jeu $\uR^{(j)}$ en le disposant de manière
verticale.
$$
\uR^{(j)} : \quad  R^{(j)}_i = X_j^{d_i-1}X_{i+1} \quad
\left\{ \begin {array} {l}
X_j^{d_1-1} X_2 \\
X_j^{d_2-1} X_3 \\
\quad \vdots \\
X_j^{d_{n-1}-1} X_n \\
X_j^{d_n-1} X_1 \\
\end {array}
\right.
$$
En un premier temps, on s'alloue des indéterminées
\footnote{Les symboles $r_1,r_2,r_3$ qui apparaissent, et qui désignent des indéterminées,
n'ont rien à voir avec les mineurs principaux de la matrice $\Jac_\Un(\uX^D - \uR^{(j)})$
figurant dans le dernier corollaire.}  
$$
\bsp = (p_1, p_2,p_3), \qquad  \bsq = (q_1, q_2,q_3), \qquad
\bsr = (r_1, r_2,r_3), \qquad \bss = (s_1, s_2,s_3)
$$
et on considère provisoirement le système:
$$
\bsp\uX^D + \bsq\uR^{(1)} + \bsr\uR^{(2)} + \bss\uR^{(3)} 
$$
c'est-à-dire 
$$
\left\{
\begin {array} {*{7}c}
p_1X_1^{d_1} &+& q_1X_1^{d_1-1}X_2 &+& \cercle{$r_1$}X_2^{d_1-1}X_2  &+& s_1X_3^{d_1-1}X_2 \\
p_2X_2^{d_2} &+& q_2X_1^{d_2-1}X_3 &+& r_2X_2^{d_2-1} X_3 &+& \cercle{$s_2$}X_3^{d_2-1}X_3 \\
p_3X_3^{d_3} &+& \cercle{$q_3$}X_1^{d_3-1}X_1  &+& r_3X_2^{d_3-1}X_1 &+& s_3X_3^{d_3-1}X_1 \\
\end {array}
\right.
$$
Nous avons entouré d'un petit cercle 3 indéterminées. Pourquoi? Parce qu'elles
sont spéciales et tellement spéciales que nous allons les spécialiser à $0$
i.e. nous en débarrasser. Ce qui conduit au système~$\uPred$ donné au début.
Nous pouvons annoncer le résultat.

\begin {theo} \leavevmode
\label{omegaresPredGr2}

\begin {enumerate}[\rm i)]
\item  
On a $\Gr(\omegaresPred) \geqslant 2$. De manière plus précise, $n+1$ évaluations bien
précises suffisent:
$$
\Gr\Bigl(
\omegaresPred(X^\alpha),\quad  
X^\alpha 
\in\{X^\emouton, X_1^\delta, \dots, X_n^\delta\}
\Bigr) \geqslant 2
$$  

\item
Pour tout système $\uP$ couvrant le jeu $\uPred$, on a $\Gr(\omegaresP) \geqslant 2$.
\end{enumerate}  
\end {theo}

\begin {proof} 
Pour simplifier les notations, nous prenons $n=3$.  

\medskip
i)  
Nous allons nous limiter au cas où l'anneau de base $\bfk$ (sur lequel sont
montées les 9 indéterminées $\bsp, \bsq, \bsr, \bss$) est un anneau intègre à pgcd,
auquel cas il en est de même de $\bfk[\bsp, \bsq, \bsr, \bss]$. Pour une
suite~$\ua$ d'un anneau intègre à pgcd, il y a équivalence entre
$\Gr(\ua) \geqslant 2$ et le fait que les éléments de $\ua$ sont premiers entre eux.
Nous allons donc montrer que les 4 évaluations sont premières entre elles i.e.
pour $g \in \bfk[\bsp, \bsq, \bsr, \bss]$:  
$$
g \text{ divise }  \omegaresPred(X^\alpha) \text{ pour }
X^\alpha 
\in\{X^\emouton, X_1^\delta, X_2^\delta, X_3^\delta\}
\quad \Longrightarrow\quad
g \text{ est inversible}
\leqno (\star)
$$
On rappelle que $h := \omegaresPred(X^\emouton)$ est un polynôme
homogène en $\bsp, \bsq, \bsr,\bss$; son degré est $\widehat d_1-1 +
\widehat d_2-1 + \widehat d_3-1$ mais c'est sans importance ici. Ce
qui est important, c'est qu'il est régulier et même primitif puisqu'il
contient le monôme $p_1^{\widehat d_1-1} p_2^{\widehat d_2-1}
p_3^{\widehat d_3-1}$. En conséquence, comme $g$ divise $h$, on peut
déjà annoncer que $g$ est un polynôme homogène en $\bsp, \bsq,
\bsr,\bss$. Nous l'écrivons $g = g(\bsp,\bsq,\bsr,\bss)$. Comme
$h(\bsp,0,0,0) = p_1^{\widehat d_1-1} p_2^{\widehat d_2-1}
p_3^{\widehat d_3-1}$, on a $g(\bsp,0,0,0) \mid p_1^{\widehat d_1-1}
p_2^{\widehat d_2-1} p_3^{\widehat d_3-1}$ mais il va falloir réaliser
beaucoup mieux pour obtenir $g$ inversible.

\medskip

\noindent
$\bullet$
Spécialisons en les $r_k = s_\ell = 0$. Alors $\uPred$ se spécialise
en $\uQ$:
$$
\uQ := \bsp\uX^D + \bsq\uR^{(1)}
$$
et les deux premières relations de divisibilité 
de l'hypothèse $(\star)$ se spécialisent en:
$$
g(\bsp,\bsq,0,0) \text{ divise}\quad
\left|
\begin {array} {lcl}
\omegaRes{\uQ}(X^\emouton) &=&
p_1^{\widehat d_1-1} p_2^{\widehat d_2 -1} p_3^{\widehat d_3-1} =
p_1^\sbullet p_2^\sbullet p_3^\sbullet 
\\[3mm]
\omegaRes{\uQ}(X_1^\delta) &=&
p_1^0 p_2^{\widehat d_2 -r_2} p_3^{\widehat d_3-r_3}
\overline{q}_1^{r_1-1} \overline{q}_2^{r_2-1} \overline{q}_3^0 =
\pm p_2^\sbullet p_3^\sbullet q_1^\sbullet q_2^\sbullet
\\
\end {array}
\right.
$$
En fait, la véritable valeur des exposants importe peu: ce qui compte, c'est le
$p_1^0$ et le $q_3^0$. Et c'est pour cette raison que dans $\uPred$, on s'est
permis de se débarrasser de $q_3$ i.e. $q_3$, c'est~$0$. Bilan:
$$
g(\bsp,\bsq,0,0) \text{ divise}\quad p_2^{\widehat d_2 -1} p_3^{\widehat d_3-1} =
p_2^\sbullet p_3^\sbullet
$$

\noindent
$\bullet$
Et de la même manière (encore une fois, peu importe les exposants):
$$
g(\bsp,0,\bsr,0) \text{ divise}\quad p_1^{\widehat d_1 -1} p_3^{\widehat d_3-1} = p_1^\sbullet p_3^\sbullet
\qquad\qquad
g(\bsp,0,0,\bss) \text{ divise}\quad p_1^{\widehat d_1 -1} p_2^{\widehat d_2-1} = p_1^\sbullet p_2^\sbullet
$$
A fortiori, $g(\bsp,0,0,0)$ divise les 3 monômes
premiers entre eux:
$$
g(\bsp,0,0,0) \text{ divise}\quad p_2^\sbullet p_3^\sbullet,
\quad p_1^\sbullet p_3^\sbullet,
\quad \text{ et } 
\quad p_1^\sbullet p_2^\sbullet
$$
On en déduit que $g(\bsp,0,0,0)$ est inversible donc un polynôme de
degré $0$.  Mais $g$ étant un polynôme homogène d'un certain degré
$d$, le polynôme $g(\bsp,0,0,0)$ est homogène de même degré $d$.
Bilan: $d = 0$, $g$ est constant donc $g = g(\bsp,0,0,0)$ est
inversible, ce qu'il fallait démontrer.

\bigskip

ii)  Résulte du théorème qui suit. L'examen de la preuve de ce théorème
apporte la précision suivante sur le nombre d'évaluations:
$$
\Gr\Bigl(
\omegaresP(X^\alpha),\quad  
X^\alpha \in\{X^\emouton, X_1^\delta, X_2^\delta, X_3^\delta\}
\Bigr) \geqslant 2
$$

\end {proof}

\begin {theo} [Contrôle de la profondeur par la propriété de couverture] \leavevmode
\label{GrViaPcouvreQ} 

Soit $\uP$ un système couvrant un système $\uQ$ (cf définition \ref{JeuCouvrant}). 
On a alors $\Gr(\omegaRes{\uP}) \geqslant \Gr(\omegaRes{\uQ})$, dont la véritable
signification est:
$$
\Gr(\omegaRes{\uQ}) \geqslant k \quad\Longrightarrow\quad   \Gr(\omegaRes{\uP}) \geqslant k
$$
En particulier $\Gr(\omegaRes{\uQ}) \geqslant 2 \Longrightarrow  \Gr(\omegaRes{\uP}) \geqslant 2$.
\end {theo}

\index{théorème!de contrôle de la profondeur par!2@la propriété de couverture}%

\begin {proof} \leavevmode

Par définition de $\uP$ couvre $\uQ$, l'anneau des coefficients de
$\uQ$ est un anneau de polynômes sur un anneau~$\bfk$. Désignons par
$\uP^\gen$ le système générique de format $D$ au dessus de~$\bfk$.
Ainsi $\uP^\gen$ est à coefficients dans $\bfk[\uT]$ où les
indéterminées~$T_\ell$ sont allouées aux coefficients des $P^\gen_i$
et les coefficients non nuls des polynômes $\uQ_i$ sont certaines
indéterminées parmi les $T_\ell$ . Notons $\uU \subset \uT$ celles ne
figurant \emph {pas} dans les~$Q_i$.

\medskip 
Pour un monôme $X^\alpha$ de degré $\delta$,
soient:
$$  
F_\alpha = \omegaRes{\uP^\gen}(X^\alpha),\qquad
H_\alpha = \omegaRes{\uQ}(X^\alpha), \qquad
G_\alpha = \omegaRes{\uP}(X^\alpha)
$$
Posons $s = \sum_i (\widehat d_i-1)$.  Nous allons appliquer le
théorème~\ref{ControleProfondeur}.  du contrôle de la profondeur par
les composantes homogènes dominantes à la famille $(G_\alpha)$ de
polynômes de degré $\le s$.  Le polynôme $F_\alpha \in \bfk[\uT]$ est
homogène de degré $s$ et $H_\alpha$ est obtenu à partir de~$F_\alpha$
en spécialisant les variables de $\uU$ à 0 tandis que $G_\alpha$ est
obtenu en spécialisant les variables de $\uU$ en des éléments
quelconques de~$\bfk$.  Nous sommes alors en mesure d'appliquer le
lemme~\ref{BabyHmgTrick} qui suit.  Celui-ci fournit le fait que
$H_\alpha = \omegaRes{\uQ}(X^\alpha)$ est la composante homogène de
degré $s$ de $G_\alpha = \omegaRes{\uP}(X^\alpha)$ et c'est aussi sa
composante homogène dominante, \emph {y compris si
  $\omegaRes{\uQ}(X^\alpha)$ est nul}.

\smallskip

Le résultat est alors conséquence du théorème~\ref{ControleProfondeur}.
\end {proof}

\medskip

Donnons tout de suite comme corollaire le résultat suivant prouvé d'une autre
manière en utilisant l'injectivité de ${\uomegares}_{,\uP^\gen}$
(cf théorème~\ref{omegaresInjectiveDoncCramer}) et sa conséquence
$\Gr(\omegaRes{\uP^\gen}) \ge 2$ (point i) du théorème~\ref{omegaresGr2}).

\begin {coro}
\label{omegaresPgenGr2}
On a $\Gr(\omegaRes{\uP^\gen}) \ge 2$ où $\uP^\gen$ désigne le système générique.
\end {coro}

\begin {proof}
Le système $\uQ:=\uPred$ vérifie $\Gr(\omegaRes{\uQ}) \ge 2$ d'après le théorème~\ref{omegaresPredGr2}.
Et le système $\uP:=\uP^\gen$ couvre $\uQ$. On applique alors le théorème \ref{GrViaPcouvreQ}.
\end {proof}  

\medskip

\begin {lem} [Baby homogeneous trick] \leavevmode
\label{BabyHmgTrick}
  
Soit $F\in\bfk[\uT]$ un polynôme à plusieurs variables sur un anneau~$\bfk$,
homogène de degré $s$ et $\uU$ un sous-ensemble des variables de $\uT$.
On définit deux polynômes de $\bfk[\uT\setminus \uU]$: d'une part
$G$ obtenu à partir de $F$ en spécialisant de manière quelconque les
variables de $\uU$ et d'autre part $H$ obtenu en spécialisant à~$0$.

\medskip
Alors $H$ est homogène de degré $s$, c'est la composante homogène de degré $s$ de
$G$ et aussi sa composante homogène dominante \og en degré $s$\fg{} (\emph {même si $H$ est nul}).
\end {lem}

\begin {proof}

Pour faire simple, supposons que $\uT$ soit constitué de 3 variables
notées $a,b,c$ et que $\uU$ soit réduit à $\{c\}$.  La spécialisation
qui intervient est donc celle de $c$:
$$
G = F(a,b,\gamma) \quad\text {avec}\quad \gamma \in \bfk,
\qquad
H = F(a,b,0)
$$
Le polynôme $H$ est bien sûr homogène de degré $s$. Ecrivons:
$$
F = \sum_{i+j+k = s} \kern -7pt \lambda_{ijk}\, a^ib^jc^k = F(a,b,0) + R(a,b,c) 
\qquad \text{avec} \qquad
R = \sum_{\substack{i+j+k = s\\k \ge 1\\}} \kern -7pt \lambda_{ijk}\, a^ib^jc^k
$$
On a donc
$$
G = F(a,b,0) + R(a,b,\gamma) = H + R(a,b,\gamma)  
$$
Le polynôme $H$ étant homogène de degré $s$ et $R$ de degré $< s$, on obtient
que $H$ est la composante homogène dominante (en degré $s$) de $G$.
\end {proof}

\cleardoublepage

\section{Le résultant : définitions à la MacRae}
\label{ChapMacRaeDefResultant}

L'objectif est de montrer que les divers modules de MacRae de rang 0
que nous avons attachés à un système $\uP$ ont les mêmes invariants de
MacRae. En fait, nous allons montrer un résultat plus précis: non
seulement, les invariants de MacRae (idéaux monogènes de l'anneau
$\bfA$ des coefficients) sont identiques, mais tous leurs générateurs
privilégiés $\omegaRes{\uP}(\nabla)$ et les $\calR_d(\uP)$ que nous
avons dégagés dans le chapitre précédent
(cf.~théorème~\ref{PoidsNormalisationMacRae}) ne font qu'un: c'est le
résultant de $\uP$.

\subsection{\'Egalité des invariants de MacRae des
$\bfA$-modules $\bfB'_d(\protect\uP)_{d\ge\delta}$ de rang $0$}

On rappelle que le quotient $\bfB'$ défini par 
$$
\bfB' = \bfA[\uX]/\langle \uP,\nabla\rangle
$$
est indépendant du déterminant bezoutien choisi $\nabla$ de $\uP$. La
considération de $\bfB'$ permet de traiter les cas $d = \delta$ et
$d \geqslant \delta + 1$ de manière uniforme: $\bfB'_d$ est le module
$\bfB'_\delta$ pour $d = \delta$ et $\bfB_d$ pour
$d \geqslant \delta+1$.

\begin{theo}[\'Egalité des invariants de MacRae] \label{MacRaeEqualities} \leavevmode

Pour un système $\uP$ quelconque, on a les égalités 
$$
\omega_{\res}(\nabla) \ = \ 
\calR_{\delta+1}
\qquad \text{ et } \qquad 
\forall\, d \geqslant \delta+1, \quad 
\calR_d  = \calR_{d+1}
$$
En particulier, pour $\uP$ régulière, 
on dispose des égalités des invariants de MacRae :
$$
\MacRae(\bfB'_{\delta}) = \MacRae(\bfB_{\delta+1})
\qquad \text{ et } \qquad 
\forall\, d \geqslant \delta+1, \quad 
\MacRae(\bfB_d) = \MacRae(\bfB_{d+1})
$$
Ce qui se résume en : l'idéal $\MacRae(\bfB'_d)$ est indépendant de $d \geqslant \delta$.
\end{theo}

\index{invariant de MacRae}%
\index{théorème!e1@égalité des invariants de MacRae des $\bfB_d$ en degré $d\ge\delta+1$}

Vu l'importance de ce résultat, nous avons choisi d'en fournir trois preuves.
Voici un résumé succinct des stratégies, résumé dans lequel on convient
de noter $\calR_\delta$ pour $\omegares(\nabla)$.

\medskip
$\rhd$
La première preuve assure, via un argument de profondeur~2, l'égalité des
idéaux $\MacRae(\bfB'_d) = \MacRae(\bfB'_{d+1})$ pour
$d \geqslant \delta$.  Il est ensuite aisé d'obtenir l'égalité des
générateurs privilégiés via le fait qu'ils sont normalisés.

$\rhd$
La seconde possède des points communs avec la première mais se
contente d'assurer l'inclusion
$\MacRae(\bfB'_{d+1}) \subset \MacRae(\bfB'_d)$ pour
$d \geqslant \delta$.  On a donc $\calR_d \mid \calR_{d+1}$ et la
terminaison se réalise grâce au poids et à la normalisation des
générateurs privilégiés.  Le contrôle par le poids joue ici un rôle important
alors qu'il n'intervient pas dans la première preuve (puisque, dans celle-ci, l'égalité
des idéaux fournit directement $\calR_d \sim \calR_{d+1}$).

$\rhd$
La troisième est indépendante des deux premières et très différente dans l'esprit.
Elle se limite à $d \geqslant \delta+1$ et utilise la relation de divisibilité
$\det W^\sigma_{1,d} \mid \det W^\sigma_{1,d+1}$. Grâce au pgcd fort, on en
déduit que $\calR_{d} \mid \calR_{d+1}$ et on termine en utilisant
poids et normalisation.

\subsubsection*{Partie commune aux 2 premières preuves:
une suite exacte $0 \to \bfB'_d \to \bfB'_{d+1} \to G_{d+1}\to 0$.}

Par définition même des objets, il suffit de faire la preuve pour la suite $\uP$ générique,
en prenant donc $\bfA = \bfk[\indetsPi]$ où $\bfk$ est un anneau quelconque.

\smallskip
\noindent
Puisque $X_1\nabla \in \langle \uP \rangle$, on dispose de la suite exacte de $\bfA[\uX]$-modules gradués 
$$
\xymatrix @M=0.4pc{
\bfA[\uX]/\langle\uP,\, \nabla\rangle \ar[r]^-{\times X_1} & 
\big(\bfA[\uX]/\langle\uP\rangle\big)(1) \ar[r] &
\big(\bfA[\uX]/\langle X_1,\,\uP\rangle\big)(1)  \ar[r] & 
0
}
$$
En passant à la composante homogène de degré $d \geqslant \delta$, et en tenant compte
du facteur de décalage~(1), on obtient la suite \textit{exacte} de $\bfA$-modules :
$$
\xymatrix @M=0.4pc @R=0pc{
0 \ar[r] & \big(\bfA[\uX]/\langle\uP, \, \nabla \rangle \big)_d \ar[r]^-{\times X_1} & 
\big(\bfA[\uX]/\langle\uP\rangle \big)_{d+1} \ar[r] &
\big(\bfA[\uX]/\langle X_1,\,\uP\rangle\big)_{d+1} \ar[r] & 
0 \\
& = \bfB'_d  & =\bfB'_{d+1} &   & \\
}
$$
Justifions l'injectivité de l'application à gauche (de multiplication
par $X_1$) $\bfB'_d \to \bfB'_{d+1}$.

\noindent
Soit $F \in \bfA[\uX]_d$ tel que $X_1 F \in \langle \uP \rangle$.
Puisque $\uP$ est générique, on peut utiliser le point ii) de la
proposition~\ref{GeneriqueLocaliseSature} (intitulée ``En générique,
pour saturer, un seul localisé suffit'') qui fournit $F \in \uPsat_d$.
Or, d'après~\ref{MiniWiebe} (Composante homogène $\uPsat_d$ en degré
$d \ge \delta$), on a $\uPsat_d = \langle \uP, \nabla \rangle_d$.
Bilan: $F \in \langle \uP, \nabla \rangle_d$, ce qui termine la
justification.

\subsubsection*{Mise en place de la première preuve}

Nous disposons ainsi d'une suite exacte $0 \to E \to F \to G \to 0$ où
$E,F$ sont des modules de MacRae de rang 0 et $G$ un module de
présentation finie. La \og seule manière raisonnable\fg{} pour obtenir
$\MacRae(E) = \MacRae(F)$ est d'avoir
$\Gr\big(\calF_0(G)\big) \geqslant 2$. En effet, si cette inégalité
est vérifiée, $G$ est de MacRae de rang~$0$, d'invariant de MacRae
$\langle 1\rangle$, et le théorème de multiplicativité de l'invariant
de MacRae (cf.~\ref{MacRaeMultiplicativity}) fournit $\MacRae(F)
= \MacRae(E)\MacRae(G)$ d'où $\MacRae(F) = \MacRae(E)$. Il est
d'ailleurs inutile d'invoquer ce théorème de multiplicativité
puisqu'on peut le reprendre directement en charge dans ce cas
particulier.

\begin {lem} [Un cas particulier de multiplicativité de l'invariant de MacRae]
\label{MultiplicativiteMacRaeCasParticulier}
\leavevmode

\noindent
Soit $0 \to E \to F \to G \to 0$ une suite exacte où $E, F$ sont
de MacRae de rang $0$ et $\Gr(\calF_0(G)) \geqslant 2$.
Alors $\MacRae(E) = \MacRae(F)$.
\end {lem}

\begin{proof}
Notons que $G \simeq F/\iota(E)$ est bien de présentation finie
puisque $E,F$ le sont.  Soit $(a_i)$ un système générateur fini de
$\calF_0(G)$. Localisons en $a_i$; puisque $\calF_0(G) \subseteq
\Ann(G)$, on a $G=0$ sur $\bfA[1/a_i]$ donc $E \simeq F$ sur ce
localisé. En conséquence $\calF_0(E) = \calF_0(F)$ sur chaque localisé
$\bfA[1/a_i]$.  Les idéaux $\calF_0(E)$ et $\calF_0(F)$ ont donc même
pgcd (à un inversible près) sur chaque localisé.  Donc
d'après~\ref{PropProfondeur2}, ces deux pgcd sont égaux globalement sur
$\bfA$.  Ce qui s'écrit exactement $\MacRae(E) = \MacRae(F)$.
\end{proof}

Pourquoi avoir écrit la \og seule manière raisonnable\fg?  Pour la
raison suivante: pour peu que $G$ soit de MacRae (de rang 0), les
égalités $\MacRae(F) = \MacRae(E)\MacRae(G)$ et $\MacRae(E) =
\MacRae(F)$ forcent $\MacRae(G) = \langle 1\rangle$
i.e. $\Gr\big(\calF_0(G)\big) \geqslant 2$.

\medskip

C'est donc cette inégalité de profondeur $\geqslant 2$ que nous allons
assurer.  Les moyens que nous avons à notre disposition n'étant pas si
nombreux, nous avons encore une fois opté sur l'utilisation du
contrôle de la profondeur par les composantes homogènes dominantes
et d'un système ad-hoc $\uQ = \bsp\uX^D + \bsq\uR$ où $\uR$ est un jeu monomial.

Ce système $\uQ$ nécessite l'introduction de $2n-1$ indéterminées
$p_1, \dots, p_n, q_1, \dots, q_{n-1}$ sur un anneau~$\bfR$. Il est
défini de la manière suivante:
$$
Q_i = p_i X_i^{d_i} + q_i X_{i+1}^{d_i} 
\qquad \text{pour $i < n$} \qquad \text{ et } \qquad 
Q_n = p_n X_n^{d_n}
$$
Voici à quoi il va être utile.

\begin{theo}\label{TheoFittingProfondeur2}
Notons $G_d = \bfA[\uX]_d/\langle X_1,\,\uP\rangle_d$.
Soit $\uP$ un système couvrant le jeu $\uQ$, typiquement le jeu générique.
Alors, pour $d \geqslant \delta$,  les deux déterminants suivants
$$
\det W_{1,d+1}(\uP)
\qquad \det W_{1,d+1}(X_1, P_1, \dots, P_{n-1})
$$
constituent une suite \emph{régulière} de l'idéal $\calF_0(G_{d+1})$.
En conséquence $\Gr\big(\calF_0(G_{d+1})\big) \geqslant 2$.
\end{theo}

La preuve est reportée après le lemme suivant. On notera que
l'inégalité $d \geqslant \delta$ est indispensable ; non pas en ce qui concerne le
statut de régularité de la suite (elle est régulière pour tout $d$, cf
le lemme ci-dessous) mais en ce qui concerne l'appartenance du premier déterminant à l'idéal de
Fitting $\calF_0(G_{d+1})$.
Pour $d = \delta-1$, il se peut que l'idéal $\calF_0(G_{d+1}) = \calF_0(G_\delta)$ ne soit
pas de profondeur $\geqslant 2$. En voici un exemple. Le $\bfA$-module
$G_\delta \overset{\rm def.}{=} \bfA[\uX]_\delta/\langle X_1,\,\uP\rangle_\delta$ 
étant présenté par l'application linéaire:
$$
\begin{array}[t]{rcl}
\bfA[\uX]_{\delta-1} \times \bfA[\uX]_{\delta-d_1} \times \cdots 
\times \bfA[\uX]_{\delta - d_n} & \longrightarrow & \bfA[\uX]_\delta \\[0.2cm]
(V,U_1,\dots, U_n) 
& \longmapsto & 
VX_1 + \sum_{i=1}^n U_iP_i \\
\end{array}
$$
son idéal de Fitting $\calF_0(G_\delta)$ est engendré par les mineurs pleins de cette application linéaire.
Prenons un système $\uP = (L_1,
\dots, L_{n-1}, P_n)$ de format $D = (1,\dots,1,2)$. On a $\delta=1$
et le module $G_\delta$ est
présenté par la matrice carrée $n \times n$ où les $n-1$ polynômes linéaires
$L_j$ sont en colonne:
$$
\begin {bmatrix}
1      &\vdots &       &\vdots\\
0      &L_1    &\cdots & L_{n-1} \\
\vdots &\vdots &       &\vdots\\
0      &\vdots &       &\vdots \\
\end {bmatrix}
$$
Ainsi $\calF_0(G_\delta)$, engendré par le déterminant $n\times n$ ci-dessus,
n'est pas de profondeur $\geqslant 2$.

\begin{lem}
On fixe un ordre monomial < tel que $X_1 > \dots > X_n$.
Ici le degré $d$ est quelconque.

\begin{enumerate}[\rm i)]
\item 
Soit $\calM \subset \Jex_{1,d}(D)$ un sous-module monomial. 

Alors l'endomorphisme $W_{\calM}(\uQ)$ est triangulaire relativement à l'ordre monomial.
De manière précise, pour $X^\alpha \in \calM$ fixé:
$$
W_{\calM}(\uQ)(X^\alpha) 
\ \in \ 
\bigoplus_{X^{\alpha'} \leqslant X^{\alpha}} \bfA X^{\alpha'}
$$
De plus, sa diagonale est constituée de $p_i$. En particulier,
$\det W_\calM(\uQ)$ est un produit de $p_i$.

\item 
Soit $\calM' \subset \Jex_{1,d}(D')$ un sous-module monomial où $D' = (1,d_1, \dots, d_{n-1})$.
On note $\uQ'$ le jeu de format $D'$ défini par
$\uQ' = (X_1, Q_1, \dots, Q_{n-1})$.

L'endomorphisme $W_{\calM'}(\uQ')$ est triangulaire relativement à l'ordre monomial,
avec une relation d'ordre inversée par rapport à celle du point précédent \idest{}
pour $X^\alpha \in \calM'$ fixé:
$$
W_{\calM'}(\uQ')(X^\alpha) 
\ \in \ 
\bigoplus_{X^{\alpha'} \geqslant X^\alpha} \bfA X^{\alpha'}
$$
De plus, sa diagonale est constituée de $1$ ou $q_j$.
En particulier, $\det W_{\calM'}(\uQ')$ est un produit de $q_j$.

\item
Soit $\uP$ un système couvrant le jeu $\uQ$ et soit $\uP'$ le jeu $(X_1, P_1, \dots, P_{n-1})$.
Alors les deux déterminants suivants
$$
\det W_{1,d}(\uP), \qquad \det W_{1,d}(\uP')
$$
forment une suite \emph{régulière}.
\end{enumerate}
\end{lem}

\begin{proof} \leavevmode
  
i) 
Soit $X^\alpha \in \calM$ et $i = \minDiv_D(X^\alpha)$. 
Si $i = n$, alors $W_\calM(\uQ)(X^\alpha) = p_n X^\alpha$ ; sinon 
$$
W_\calM(\uQ)(X^\alpha) \ = \ 
\pi_{\calM}\Big(\dfrac{X^\alpha}{X_i^{d_i}} Q_i\Big) \ = \ 
p_i \, X^\alpha \ + \ q_i \, \pi_{\calM}(X^{\alpha'}) 
\qquad
\text{avec $X^{\alpha'} = \dfrac{X^\alpha}{X_i^{d_i}}X_{i+1}^{d_i}$}
$$
Comparons $X^{\alpha'}$ et $X^\alpha$. 

\noindent
On a $X_i > X_{i+1}$ puis 
$X_i^{d_i} > X_{i+1}^{d_i}$ et en multipliant par 
$\frac{X^\alpha}{X_i^{d_i}}$, on obtient $X^{\alpha} > X^{\alpha'}$.

\medskip
ii) 
Soit $X^\alpha \in \calM'$ et $i = \minDiv_{D'}(X^\alpha)$. 
Si $i = 1$, alors $W_{\calM'}(\uQ')(X^\alpha) = X^\alpha$ ; sinon
$$
W_{\calM'}(\uQ')(X^\alpha) \ = \ 
\pi_{\calM'}\Big(\dfrac{X^\alpha}{X_i^{d_{i-1}}}Q_{i-1}\Big) \ = \ 
q_{i-1} \, X^\alpha \ + \ p_{i-1} \, \pi_{\calM'}(X^{\alpha'}) 
\qquad
\text{avec $X^{\alpha'} = \dfrac{X^\alpha}{X_i^{d_{i-1}}}X_{i-1}^{d_{i-1}}$}
$$
On a $X_{i-1} > X_i$ puis 
$X_{i-1}^{d_{i-1}} > X_{i}^{d_{i-1}}$ et en multipliant par 
$\frac{X^\alpha}{X_i^{d_{i-1}}}$, on obtient $X^{\alpha'} > X^\alpha$.

\medskip
iii)
Puisque $\uP$ couvre $\uQ$, l'anneau $\bfA$ des coefficients
de $\uP$ s'écrit $\bfA=\bfR[p_1,\dots, p_n,q_1, \dots, q_{n-1}]$.
En vertu de~\ref{ControleProfCompHmgDom-nequal2}, il suffit de montrer que la
suite formée des deux composantes homogènes dominantes des deux déterminants
(relativement à la structure d'anneau de polynômes de $\bfA$ sur $\bfR$) 
est régulière.

\smallskip
Ordonnons la base monomiale $\Jex_{1,d}(D)$ de manière décroissante. Alors la matrice
$W_{1,d}(\uQ)$ est triangulaire inférieure avec des $p_i$ sur la
diagonale. De ce fait, la partie supérieure stricte de la matrice $W_{1,d}(\uP)$ est
constituée d'éléments de $\bfR$:
$$
W_{1,d}(\uQ) =
\begin {bmatrix}
p_\sbullet &0         &0 \\
*         &p_\sbullet &0 \\
*         &*         &p_\sbullet \\
\end {bmatrix}
\qquad \qquad
W_{1,d}(\uP) =
\begin {bmatrix}
p_\sbullet &r         &r' \\
*         &p_\sbullet &r'' \\
*         &*         &p_\sbullet \\
\end {bmatrix}
$$
En conséquence, la composante homogène dominante de $\det W_{1,d}(\uP)$ est 
le produit de la diagonale de~$W_{1,d}(\uP)$, qui est un produit de~$p_i$.

\smallskip
Comme $\uP'$ couvre $\uQ'$, on peut faire un raisonnement analogue pour $(\uQ', \uP')$ modulo la permutation inférieur $\leftrightarrow$ supérieur 
$$
W_{1,d}(\uQ') =
\begin {bmatrix}
1         &*         &* \\
0         &q_\sbullet &* \\
0         &0         &q_\sbullet \\
\end {bmatrix}
\qquad \qquad
W_{1,d}(\uP') =
\begin {bmatrix}
1         &*         &* \\
r         &q_\sbullet &* \\
r'        &r''       &q_\sbullet \\
\end {bmatrix}
$$
Ainsi la composante homogène dominante $\det W_{1,d}(\uP')$ est
le produit de la diagonale, produit de~$q_j$.

\medskip
Bilan: les composantes homogènes dominantes de $\det
W_{1,d}(\uP)$ et $\det W_{1,d}(\uP')$  forment une suite régulière et il en est de
même des deux déterminants.
\end{proof}


\begin{proof} [Preuve du théorème \ref{TheoFittingProfondeur2}] \leavevmode

Par définition, l'idéal de Fitting $\calF_0(G_d)$
du $\bfA$-module $G_d \overset{\rm def.}{=} \bfA[\uX]_d/\langle X_1,\,\uP\rangle_d$ 
est engendré par les mineurs pleins de l'application linéaire:
$$
\begin{array}[t]{rcl}
\bfA[\uX]_{d-1} \times \bfA[\uX]_{d-d_1} \times \cdots 
\times \bfA[\uX]_{d - d_n} & \longrightarrow & \bfA[\uX]_d \\[0.2cm]
(V,U_1,\dots, U_n) 
& \longmapsto & 
VX_1 + \sum_{i=1}^n U_iP_i \\
\end{array}
$$
Le degré critique de $D'$ est $\delta' = \delta-(d_n-1)$.  Supposons
$d\geqslant \delta+1$, a fortiori $d \geqslant \delta'+1$.  En
conséquence, $\Jex_{1,d}(D) = \Jex_{1,d}(D') = \bfA[\uX]_d$, donc les
deux déterminants du point iii) du lemme précédent sont des mineurs
pleins de l'application linéaire ci-dessus et donc
appartiennent à~$\calF_0(G_d)$. D'où $\Gr\big(\calF_0(G_d)\big) \geqslant 2$
pour $d\geqslant \delta+1$. En tenant compte de $d \leftrightarrow d+1$,
c'est ce qu'il fallait démontrer.
\end{proof}

\begin{proof} [Première preuve du théorème \ref{MacRaeEqualities}] \leavevmode

Le théorème \ref{TheoFittingProfondeur2} fournit l'égalité des idéaux
$\MacRae(\bfB'_d) = \MacRae(\bfB'_{d+1})$ pour
$d \geqslant \delta$. En rappelant que le terrain est générique, les
deux générateurs privilégiés diffèrent multiplicativement d'un
inversible $\lambda$ de l'anneau des coefficients $\bfA
= \bfk[\indetsPi]$. Puisque $\lambda$ est inversible et homogène,
c'est un inversible de $\bfk$. La spécialisation en le jeu étalon
fournit $\lambda=1$.
\end {proof}

\subsubsection*{Seconde preuve}

\begin{proof} [Seconde preuve du théorème \ref{MacRaeEqualities}] \leavevmode

Dans la suite exacte courte de $\bfA$-modules mise en place dans la partie
commune, les modules $\bfB'_d$ et $\bfB'_{d+1}$  sont librement résolubles de rang $0$
d'après~\ref{ResolutionQuotients}.
Il est alors classique  (cf. par exemple~\cite[Annexe~A, th 4.1, page 208]{Tete})
que le module à l'extrémité droite l'est aussi.
Ainsi tous les modules de la suite exacte sont de MacRae de rang~$0$. 
Par multiplicativité de l'invariant de MacRae (cf.~\ref{MacRaeMultiplicativity}),
on en déduit que 
$\MacRae(\bfB'_{d+1}) \subset \MacRae(\bfB'_{d})$.

\medskip
\noindent
Terminons cette preuve en distinguant les cas $d = \delta$ et $d \geqslant \delta+1$.

\noindent
$\rhd$ Pour $d = \delta$, l'inclusion $\MacRae(\bfB'_{d+1}) \subset \MacRae(\bfB'_{d})$
s'écrit 
$$
\MacRae(\bfB_{\delta+1}) = \langle \calR_{\delta+1} \rangle
\quad  \subset \quad 
\MacRae(\bfB'_{\delta}) = \langle \omega_\res(\nabla)\rangle
$$
Il~y~a donc un multiplicateur $\lambda \in \bfA$ 
tel que $\calR_{\delta+1} = \lambda\,\omega_\res(\nabla)$.
Or ces générateurs $\calR_{\delta+1}$ et $\omega_\res(\nabla)$
ont même poids en chaque $P_i$ d'après~\ref{PoidsNormalisationMacRae}, 
donc $\lambda \in \bfk$.
En spécialisant $\uP$ en le jeu étalon, $\lambda$ n'est pas affecté, et on trouve que 
$1 = \lambda \times 1$ donc $\lambda = 1$.

\smallskip
\noindent
$\rhd$ Pour $d \geqslant \delta+1$, l'inclusion $\MacRae(\bfB'_{d+1}) \subset \MacRae(\bfB'_{d})$
s'écrit 
$$
\MacRae(\bfB_{d+1}) = \langle \calR_{d+1} \rangle
\quad  \subset \quad 
\MacRae(\bfB_{d}) = \langle \calR_{d} \rangle
$$
On termine comme précédemment à l'aide des propriétés de poids et de normalisation 
de $\calR_d$ (cf.~\ref{PoidsNormalisationMacRae}).
\end{proof}

\subsubsection*{Preuve de  $\calR_d(\uP) = \calR_{d+1}(\uP)$
via la caractérisation \og pgcd fort \fg{} du MacRae}

Cette preuve n'utilise pas le fait que $\uP$ est générique, seulement
la propriété que $\uP$ couvre le jeu étalon généralisé $(p_1X_1^{d_1}, \dots, p_nX_n^{d_n})$.
Ainsi, $\calR_d$ est le pgcd fort des mineurs $(\det
W_{1,d}^\sigma)_{\sigma \in \Sigma_2}$, cf.~\ref{MacRaePGCDfort}.

\begin{proof} [Troisième preuve du théorème \ref{MacRaeEqualities} pour $d \geqslant \delta+1$]  \leavevmode

Nous allons utiliser un cas particulier du résultat~\ref{InclusionPGCD}, à savoir :

\begin{quote}
\it 
Dans un anneau commutatif, supposons disposer de deux 
familles finies $\ua = (a_j)$, $\ub = (b_j)$ indexées
par le même ensemble d'indices, chacune ayant un pgcd fort; si $a_j \mid b_j$
pour tout $j$, alors $\mathrm{pgcd}(\ua) \mid \mathrm{pgcd}(\ub)$.
\end{quote}

Pour $d \geqslant \delta+1$ et $\sigma\in\fS_n$, on a $\det
W_{1,d}^\sigma(\uP) \mid \det W_{1,d+1}^\sigma(\uP)$
(confer~\ref{DetW1dDiviseDetW1dplus1}).  Soit
$\Sigma_2 \subset \fS_n$ l'ensemble des transpositions $(i,n)$
pour $1 \leqslant i \leqslant n$.  
L'application du résultat général ci-dessus aux
deux familles $\big(\det
W_{1,d}^\sigma(\uP)\big)_{\sigma \in \Sigma_2}$ et $\big(\det
W_{1,d+1}^\sigma(\uP)\big)_{\sigma \in \Sigma_2}$ fournit la relation
de divisibilité entre leur pgcd fort : on a $\lambda\,\calR_{d}(\uP) = \calR_{d+1}(\uP)$
avec $\lambda \in\bfA$.  Pour des raisons de
poids, et du fait de la présence du monôme $p_1^{\widehat d_1} \cdots
p_n^{\widehat d_n}$ dans chacun des deux invariants de MacRae, on a
$\lambda =1$.
\end{proof}

\subsection{Le résultant comme générateur normalisé de l'invariant de MacRae}

La définition suivante est légitimée par l'égalité des différents générateurs de MacRae
que nous venons d'établir en~\ref{MacRaeEqualities}.
Elle~a lieu en deux temps : pour une suite $\uP$ générique, puis pour une suite $\uP$ quelconque.
\'Evidemment, pour~$\uP$ régulière, 
il est important que les deux points de vue coïncident; c'est bien le cas, car 
l'invariant de MacRae se spécialise (cf.~\ref{ExtensionScalaires}).

\index{invariant de MacRae}%

\begin{defn} [Résultant] \label{DefResultant} \leavevmode

Le résultant de $\uP$ est le scalaire, noté $\Res(\uP)$, défini de la manière suivante:

\begin{itemize}
\item 
pour $\uP$ générique, c'est le générateur normalisé (au sens de~\ref{PoidsNormalisationMacRae})
d'un des idéaux suivants de $\bfA = \bfk[\indetsPi]$
$$
\MacRae(\bfB'_{\delta}) 
\qquad \text{ ou } \qquad 
\MacRae(\bfB_d)
\quad \text{pour } d \geqslant \delta+1
$$
à savoir 
$$
\omega_\res(\nabla) 
\qquad \text{ ou } \qquad 
\calR_d
\quad \text{pour } d \geqslant \delta+1
$$
\item 
pour $\uP$ quelconque, c'est la  spécialisation du résultant générique.
\end{itemize}
\end{defn}

\index{résultant}%

\bigskip
\label{IdealCofacteur}

\noindent
Pour une suite $\uP$ quelconque, on dispose des idéaux cofacteurs $\fb_d$ définis en
\ref{ResumeSpecialisationCasGenerique}
$$
\calF_0(\bfB'_\delta) \ = \ 
\Res(\uP) \, \fb_\delta
\qquad \text{ et } \qquad 
\forall\, d \geqslant \delta+1,\quad 
\calF_0(\bfB_d) \ = \ 
\Res(\uP) \, \fb_d
$$
Par définition, les idéaux de Fitting qui interviennent ci-dessus s'expriment comme
des idéaux déterminantiels de l'application de Sylvester $\Syl = \Syl(\uP)$
et ces deux égalités s'écrivent aussi 
$$
\DVect_{s_\delta}(\Syl_\delta)\text{-évalué-en-$\nabla$}
\ = \ 
\Res(\uP) \, \fb_\delta
\qquad \text{ et } \qquad 
\forall\, d \geqslant \delta+1,\quad 
\calD_{s_d}(\Syl_d)
\ = \ 
\Res(\uP) \, \fb_d
$$
Dans chaque idéal de gauche, il y a un scalaire privilégié (avec un nom de baptème) :
$$
\text{$\omega(\nabla)$ pour $d = \delta$
\qquad \qquad 
$\det W_{1,d}$ pour $d \geqslant \delta+1$ 
}
$$
Il s'agit d'un cas particulier pour $\sigma = \Id_n$ de deux familles
de scalaires indexées par $\sigma \in \fS_n$ :
$$
\text{$\omega^\sigma(\nabla)$ pour $d = \delta$
\qquad \qquad 
$\det W^\sigma_{1,d}$ pour $d \geqslant \delta+1$ 
}
$$
Tous ces scalaires privilégiés sont donc des multiples du résultant.

\index{idéal!cofacteur $\fb_d$}%

\bigskip

Nous énonçons ci-dessous en termes du \textit{résultant} les
propriétés du résumé page~\pageref{ResumeSpecialisationCasGenerique}
du chapitre~\ref{ChapMacRaeForP}.

\begin{prop}[Spécialisation, régularité, poids et normalisation, pgcd fort] 
\label{ProprietesBasiquesResultant}

\leavevmode
\begin{enumerate}[\rm i)]
\item 
Pour le jeu étalon généralisé, on a 
$$
\Res(p_1X_1^{d_1}, \dots, p_nX_n^{d_n}) 
\ = \ 
p_1^{\widehat d_1} \cdots p_n^{\widehat d_n}
$$
En particulier, $\Res(X_1^{d_1}, \dots, X_n^{d_n}) = 1$.

\item 
Le résultant d'une suite régulière est un élément régulier.

\item 
Le résultant de la suite générique $\uP = (P_1, \dots, P_n)$ est un polynôme
en les coefficients de $P_1, \dots, P_n$ 
qui contient le monôme $p_1^{\widehat d_1} \cdots p_n^{\widehat d_n}$ 
(où $p_i$ est le coefficient en $X_i^{d_i}$ de $P_i$) ;
pour tout $i$, il est homogène en les coefficients de $P_i$ de poids $\widehat d_i$.

En particulier, pour toute suite $\uP$, et pour tous $a_1, \dots, a_n \in \bfA$, on a 
$$
\Res(a_1 P_1, \dots, a_n P_n) \ = \ 
a_1^{\widehat d_1} \cdots a_n^{\widehat d_n}\, \Res(\uP) 
$$

\item
En désignant par $\Sigma_2 \subset \fS_n$
n'importe quelle partie telle que $\Sigma_2(n) = \{1..n\}$ (par
exemple l'ensemble des transpositions $(i,n)$ pour 
$1 \leqslant i \leqslant n$), le
résultant d'une suite $\uP$ \emph{couvrant le jeu étalon généralisé}
est le pgcd fort normalisé

\begin{itemize}
\item 
de $\big(\det W_{1,d}^\sigma\big)_{\sigma \in \Sigma_2}$ pour $d \geqslant \delta+1$ fixé.

\item 
de $\big(\omega^\sigma(\nabla)\big)_{\sigma \in \Sigma_2}$.
\end{itemize}
\end{enumerate}
\end{prop}

\begin{proof}
i)  Cf. la proposition \ref{BdJeuEtalonGen}.

ii) L'invariant de MacRae d'un module de MacRae de rang 0 est par définition un scalaire régulier.

iii) Voir le théorème \ref{PoidsNormalisationMacRae}.

iv) Cf. la proposition \ref{MacRaePGCDfort}.
\end{proof}

\subsection{L'exemple du format $D = (1,\dots,1,2)$}

Bien qu'il ne soit pas en accord avec le titre de la section, nous
glissons ici le cas linéaire.

\begin{prop}[Le jeu linéaire] \label{ResultantFormesLineaires}
Soit $\uL = (L_1, \dots, L_n)$ une suite de formes linéaires.
Alors 
$$
\Res(\uL) = \det \dsL
$$ 
où 
$\dsL \in \bbM_n(\bfA)$ est \textit{la} matrice donnant les coefficients des formes linéaires:
$$
\begin{bmatrix} 
L_1 & \cdots & L_n 
\end{bmatrix} 
\ = \ 
\begin{bmatrix} 
X_1 & \cdots & X_n 
\end{bmatrix} 
\dsL
$$
En particulier, pour le jeu étalon linéaire permuté, on a 
$$
\Res(X_{\sigma(1)}, \dots, X_{\sigma(n)}) 
\ = \ 
\text{signature de la permutation $\sigma$}
$$
\end{prop}

\begin{proof}
Nous allons fournir deux arguments en exploitant le fait que le format $D = (1, \dots, 1)$
est très particulier, confer~\ref{CasParticuliers}.

\smallskip

\noindent
$\rhd$ Utilisation de $\Res(\uL) = \calR_{\delta+1}$.
En raison du format $D$, on a $\calR_{\delta+1} = \det W_{1,\delta+1}$.
Le degré critique $\delta$ vaut $0$ 
et la matrice $\dsL$ coïncide clairement avec la matrice $W_{1,\delta+1}=W_{1,1}$.
Donc $\Res(\uL) = \det \dsL$.

\smallskip
\noindent
$\rhd$ 
Utilisation de $\Res(\uL) = \omega_\res(\nabla)$.
En raison du format $D$, on a $\omega_\res = \omega$ et la forme linéaire $\omega$ 
est l'identité de $\bfA$. 
De plus, un déterminant bezoutien $\nabla$ est $\det \dsL$. Par conséquent, 
$\Res(\uL) = \id_\bfA(\det \dsL) = \det \dsL$.

\medskip

Pour le jeu $(X_{\sigma(1)}, \dots, X_{\sigma(n)})$,
la matrice $\dsL$ est la matrice de la permutation $\sigma$, dont 
le déterminant est la signature.
D'où l'égalité annoncée.
\end{proof}

\subsubsection {Exemple $D=(1,\dots,1,2)$ pour lequel on a l'égalité des trois formes
                $\omegares = \omega = \evalxi$}

Ce format $D = (1, \dots, 1, 2)$, de degré critique $\delta = 1$, a
déjà fait l'objet de la proposition~\ref{CasParticuliers}.  Sa
particularité réside dans le fait que la forme linéaire $\omega_\res$
est égale à~$\omega$. Rappelons que cela provient de la nullité de
$\rmK_{2,\delta}$ conduisant au fait que le complexe
$\rmK_{\sbullet,\delta}$ est réduit à $0 \to \rmK_{1,\delta}
\xrightarrow{\ \Syl_\delta\ } \bfA[\uX]_\delta$ avec $\dim
\rmK_{1,\delta} = \dim\bfA[\uX]_\delta + 1 = n+1$.

Le résultant de $\uP$ s'explicite alors sous la forme $\omega(\nabla)$
où $\nabla$ est un déterminant bezoutien de~$\uP$.  Puisque ce format
est également celui du premier cas d'école $(1, \dots, 1,e)$, on \emph {doit
pouvoir retrouver} l'égalité $\Res(\uP) = P_n(\uxi)$ où $\uxi =
(\xi_1, \dots, \xi_n)$ est le point d'évaluation défini par les
$n-1$ polynômes linéaires $P_1, \dots, P_{n-1}$ (cf le premier cas
d'école~\ref{soussectionPremierCasEcole}):
$$
\xi_i = \det(P_1, \dots, P_{n-1}, X_i)
$$
Pour retrouver $\Res(\uP) = P_n(\uxi)$, nous allons d'abord nous
interroger sur l'égalité \boxed{\omega = \evalxi}. En utilisant
$$
\bfA[\uX]_\delta = \bfA[\uX]_1 = \bfA X_1 \oplus \cdots \oplus \bfA X_n,
\qquad\qquad
\Jex_{1,\delta} = \Jex_{1,1} = \bfA X_1 \oplus\cdots\oplus \bfA X_{n-1}
$$
et $\minDiv(X_i) = i$ pour $i < n$, on obtient, par définition de $\omega$:
$$
\omega(X_i) = \det(P_1, \dots, P_{n-1}, X_i)
\overset {\rm def.}{=} \xi_i
$$
Ainsi $\omega$ est bien la restriction à $\bfA[\uX]_1$ de la forme
linéaire \og évaluation en $\uxi$ \fg{} $\bfA[\uX]\to\bfA$.

Par exemple, pour $n=3$, avec
$$
P_1 = a_{1}X_{1} + a_{2}X_{2} + a_{3}X_{3}, \qquad
P_{2} = b_{1}X_{1} + b_{2}X_{2} + b_{3}X_{3}
$$
voici la forme linéaire $\omega : \bfA[\uX]_\delta \to \bfA$:
$$
\omega : \quad 
\begin{array}[t]{rcl}
\bfA[\uX]_1 = \bfA X_1 \oplus \bfA X_2 \oplus \bfA X_3 & \longrightarrow & \bfA \\ 
F = F_{X_1} X_1 + F_{X_2} X_2 + F_{X_3} X_3 & \longmapsto & 
\begin{vmatrix}
a_{1}& b_{1}& F_{X_1} \\ 
a_{2}& b_{2}& F_{X_2} \\ 
a_{3}& b_{3}& F_{X_3} \\ 
\end{vmatrix}
\ = \ 
\xi_1 F_{X_1} \ + \ \xi_2 F_{X_2} \ + \  \xi_3 F_{X_3}
\end{array}
$$
Venons en maintenant à la formule $\Res(\uP) = P_n(\uxi)$ du premier
cas d'école que \emph {nous entendons retrouver}.  On dispose des
égalités $\Res(\uP) \overset{\rm def.}{=} \omega_\res(\nabla)
\overset{\rm ici}{=} \omega(\nabla) \overset{\rm ici}{=}
\nabla(\uxi)$. Pour conclure, il reste à voir pourquoi $\nabla(\uxi) =
P_n(\uxi)$. Cela résulte d'un petit calcul que nous avons réalisé pour
le format $(1,\dots,1,e)$, en début de la
section~\ref{soussectionPremierCasEcole}, et qui n'est donc pas
spécifique à $D = (1, \ldots, 1, 2)$.

\bigskip

Ci-dessous, nous brodons\footnote{Broder: amplifier ou exagérer à
plaisir, selon Le Petit Robert; embellir, amplifier par des variations
et des fioritures, selon le dictionnaire de l'Académie Française.}
autour de l'égalité $\nabla(\uxi) =
P_n(\uxi)$, histoire de voir ce qu'elle dit pour $n=3$, donc dans
le cas du format $D = (1,1,2)$, avec le système $\uP$:
$$
\setlength{\tabcolsep}{2pt}
\left\{
\begin{tabular}{rcp{15cm}} 
$P_{1}$ & $=$ & $a_{1}X_{1} + a_{2}X_{2} + a_{3}X_{3}$\\ [0.1cm] 
$P_{2}$ & $=$ & $b_{1}X_{1} + b_{2}X_{2} + b_{3}X_{3}$\\ [0.1cm] 
$P_{3}$ & $=$ & $c_{1}X_{1}^{2} + c_{2}X_{1}X_{2} + c_{3}X_{1}X_{3} 
+ c_{4}X_{2}^{2} + c_{5}X_{2}X_{3} + 
c_{6}X_{3}^{2}$\\ [0.1cm] 
\end{tabular}
\right.
$$
En écrivant $P_3$ sous la forme (non unique) $P_3 = X_1 G_1 + X_2 G_2 + X_3 G_3$ 
avec les $G_i$ homogènes de degré $\delta = 1$, on obtient une 
matrice bezoutienne de $\uP$ du type suivant :
$$
\dsV =
\left[
\begin{array}{*{3}{c}}
a_{1}& b_{1}& G_1 \\ 
a_{2}& b_{2}& G_2 \\ 
a_{3}& b_{3}& G_3 \\ 
\end{array}
\right]
\quad \hbox {par exemple}\quad
\left[
\begin{array}{*{3}{c}}
a_{1}& b_{1}& c_{1}X_{1} + c_{2}X_{2} + c_{3}X_{3} \\ 
a_{2}& b_{2}& c_{4}X_{2} + c_{5}X_{3} \\ 
a_{3}& b_{3}& c_{6}X_{3} \\ 
\end{array}
\right]
$$
A droite, nous avons fait un choix particulier mais le traitement qui
vient n'en dépend pas. Il peut simplement permettre au lecteur de
réaliser quelques calculs intermédiaires pour vérification.
Notons~$\nabla$ le déterminant de~$\dsV$; celui-ci se calcule en
développant selon la dernière colonne:
$$
\nabla = \xi_1 G_1 + \xi_2 G_2 + \xi_3 G_3, \qquad
\xi_1 = \begin {vmatrix} a_2 & b_2\\ a_3 &b_3 \\ \end {vmatrix},\qquad
\xi_2 = -\begin {vmatrix} a_1 & b_1\\ a_3 &b_3 \\ \end {vmatrix},\qquad
\xi_3 = \begin {vmatrix} a_1 & b_1\\ a_2 &b_2 \\ \end {vmatrix}
$$
Via $P_3 = \sum_i X_iG_i$, on obtient ainsi $\nabla(\uxi) = P_3(\uxi)$ mais, comme nous
l'avons déjà signalé plus haut, ceci n'est ni spécifique au format $(1,1,2)$
ni au format $(1,\dots,1,2)$.

\medskip

On peut opérer quelques vérifications à la main avec le choix 
particulier de la matrice $\dsV$ ci-dessus à droite.
Le coefficient en $X_i$ de $\nabla$, noté $\nabla_{X_i}$, est obtenu 
en évaluant $\nabla$ en $X_i := 1$ et $X_j := 0$ pour $j \ne i$. 
Avec la matrice $\dsV$ ci-dessus, on obtient donc :
$$
\nabla_{X_1} \ =\  \xi_1 c_{1},
\qquad 
\nabla_{X_2} \ =\ \xi_1c_{2} + \xi_2c_{4},
\qquad
\nabla_{X_3} \ =\  \xi_1 c_{3} + \xi_2 c_{5} + \xi_3 c_{6}
$$
Ainsi, l'égalité $\Res(\uP) = \omega(\nabla) 
= \xi_1 \nabla_{X_1} + \xi_2 \nabla_{X_2} + \xi_3 \nabla_{X_3}$ s'explicite en:
$$
\Res(\uP) \ = \ 
\xi_1 (\xi_1 c_1)
\ + \ 
\xi_2 (\xi_1 c_2 + \xi_2 c_4) 
\ + \  
\xi_3 (\xi_1 c_3 + \xi_2 c_5 + \xi_3 c_6)
$$
qui n'est autre que $P_3(\uxi)$ :
$$
P_3(\uxi) \ = \ 
c_{1}\xi_{1}^{2} + c_{2}\xi_{1}\xi_{2} + c_{3}\xi_{1}\xi_{3} 
+ c_{4}\xi_{2}^{2} + c_{5}\xi_{2}\xi_{3} + 
c_{6}\xi_{3}^{2}
$$

\bigskip

\noindent
En anticipant sur la suite (cf. section~\ref{sousSectionMacaulayRecurrence}),
la formule en degré $\delta+1$ du résultant n'est pas déterminantale
au sens où l'expression $\Res(\uP) = \det W_{1,\delta+1}/\det
W_{2,\delta+1}$ fait intervenir le quotient de deux déterminants (et
pas un seul):
$$
W_{1,\delta+1} \ = \ 
\EastBordermatrix{
a_{1} & . & . & . & . & c_{1} & \Heti{X_{1}^{2}} \\ 
a_{2} & a_{1} & . & b_{1} & . & c_{2} & \Heti{X_{1}X_{2}} \\ 
a_{3} & . & a_{1} & . & b_{1} & c_{3} & \Heti{X_{1}X_{3}} \\ 
. & a_{2} & . & b_{2} & . & c_{4} & \Heti{X_{2}^{2}} \\ 
. & a_{3} & a_{2} & b_{3} & b_{2} & c_{5} & \Heti{X_{2}X_{3}} \\ 
. & . & a_{3} & . & b_{3} & c_{6} & \Heti{X_{3}^{2}} \\ 
}
\qquad  \qquad 
W_{2,\delta+1} \ = \ 
\EastBordermatrix{
\ a_{1} & \Heti{X_{1}X_{2}} 
}
$$
En voici la valeur:
{\small
$$
\begin {array} {l}
\\
\\
\\
\Res(\uP) \ = \ \\
\\
\end {array}
\begin{array}{c}
a_{1}^{3}b_{2}^{2}c_{6} -a_{1}^{3}b_{2}b_{3}c_{5} + a_{1}^{3}b_{3}^{2}c_{4} 
-2a_{1}^{2}a_{2}b_{1}b_{2}c_{6} + a_{1}^{2}a_{2}b_{1}b_{3}c_{5} + 
a_{1}^{2}a_{2}b_{2}b_{3}c_{3} - a_{1}^{2}a_{2}b_{3}^{2}c_{2} + 
a_{1}^{2}a_{3}b_{1}b_{2}c_{5} \\ - 2a_{1}^{2}a_{3}b_{1}b_{3}c_{4}  
-a_{1}^{2}a_{3}b_{2}^{2}c_{3} + a_{1}^{2}a_{3}b_{2}b_{3}c_{2} + 
a_{1}a_{2}^{2}b_{1}^{2}c_{6} -a_{1}a_{2}^{2}b_{1}b_{3}c_{3} + 
a_{1}a_{2}^{2}b_{3}^{2}c_{1} -a_{1}a_{2}a_{3}b_{1}^{2}c_{5} + \\ 
a_{1}a_{2}a_{3}b_{1}b_{2}c_{3} + a_{1}a_{2}a_{3}b_{1}b_{3}c_{2} 
-2a_{1}a_{2}a_{3}b_{2}b_{3}c_{1} + a_{1}a_{3}^{2}b_{1}^{2}c_{4} 
-a_{1}a_{3}^{2}b_{1}b_{2}c_{2} + a_{1}a_{3}^{2}b_{2}^{2}c_{1} \\ [0.1cm]
\hline \\ [-0.3cm]
a_1 
\end{array}
$$
}
A noter que le monôme de tête que l'on voit ci-dessus, après
simplification par $a_1$, n'est autre (en utilisant notre notation
habituelle $p_i$ coefficient en $X_i^{d_i}$ de $P_i$) que le fameux monôme suivant :
$$
p_1^{\widehat d_1} p_2^{\widehat d_2} p_3^{\widehat d_3} \qquad
\text{où } \quad 
\widehat d_1 = d_2d_3 = 2, \quad
\widehat d_2 = d_1d_3 = 2, \quad
\widehat d_3 = d_1d_2 = 1
$$
On peut se demander, dans le cas de bébé $D = (1,1,2)$, comment montrer directement que
$\det W_{1,\delta+1}$ est divisible par~$a_1$ et comment prouver que le quotient
(exact) est égal à $\omega(\nabla)$.

\subsection{Propriétés de divisibilité entre deux systèmes $\protect\uP$ et $\protect\uP'$
lorsque $\langle\protect\uP'\rangle \subset \langle\protect\uP\rangle$}

Pour la proposition suivante, on se donne deux formats : $D =
(d_1, \dots, d_n)$ de degré critique $\delta$ et $D' = (d'_1, \dots,
d'_n)$ de degré critique~$\delta'$ ; on 
se donne également deux systèmes de polynômes : 
$\uP$ de format~$D$ et~$\uP'$ de format~$D'$, le tout vérifiant
$\langle\uP'\rangle \subset\langle\uP\rangle$.  
On peut donc écrire
$$
P'_j = \sum_i U_{ij} P_i \quad
\hbox {où $U_{ij} \in \bfA[\uX]$ est homogène de degré $d'_j - d_i$}
$$
On définit ainsi une matrice $U \in \bbM_n(\bfA[\uX])$ telle que
$$
[P'_1, \dots, P'_n] = [P_1, \dots, P_n]\, U
$$
Dans ce qui suit, tout est donc porté par $\uP$ et la matrice $U$: à partir
de ces deux données, on peut définir $\uP' = \uP\,U$ de format $D'$ vérifiant
$\langle\uP'\rangle \subset \langle\uP\rangle$.

\medskip
Le déterminant $\det U$ est homogène de degré $\sum_j d'_j - \sum_i
d_i = \delta' - \delta$, ce qui fait que l'on peut considérer la
multiplication par $\det U$ : $\bfA[\uX]_\delta \to
\bfA[\uX]_{\delta'}$. L'objectif est de montrer que la forme linéaire
sur $\bfA[\uX]_\delta$ définie par $F \mapsto \omegaRes{\uP'}(F\det
U)$ est multiple de $\omegaRes{\uP}$ et d'en déduire en particulier
que $\Res(\uP) \mid \Res(\uP')$.


\begin {prop}[Divisibilité] \label{PremierLemmeDivisibiliteJPJ}Dans le contexte ci-dessus:

\begin{enumerate}[\rm i)]
\item
Si $\nabla_\uP$ est un déterminant bezoutien de $\uP$, alors
$\nabla_\uP\,\det U$  est un déterminant bezoutien de $\uP'$.
  
\item
Il y a un scalaire bien déterminé $c \in \bfA$ tel que
$$
\forall\, F \in \bfA[\uX]_\delta, \qquad \omegaRes{\uP'}(F\det U) = c\ \omegaRes{\uP}(F)
$$
\item
Si le scalaire $\omegaRes{\uP}(X^\emouton)$ est régulier, cette constante $c$ peut être déterminée par
$$
c = \frac{\omegaRes{\uP'}(X^\emouton \det U)} {\omegaRes{\uP}(X^\emouton)}
$$  
 
\item
On a $\Res(\uP') = c\,\Res(\uP)$, a fortiori $\Res(\uP) \mid \Res(\uP')$.
\end {enumerate}  
\end {prop}

\index{théorème!de divisibilité du résultant}%
%
%

\begin {proof} \leavevmode

\smallskip
i) Soit $\dsV$ une matrice bezoutienne de $\uP$ de déterminant $\nabla_\uP$.
Par définition, on a $\uP =\uX \dsV$ donc $\uP' = \uX (\dsV U)$.
On vérifie que $\dsV U$ est homogène si bien que $\dsV U$ est une matrice bezoutienne
de $\uP'$ et son déterminant est $\nabla_\uP\,\det U$.

\smallskip
ii) Soit $\mu$ la forme linéaire sur $\bfA[\uX]_\delta$ définie par $\mu : F  \mapsto \omegaRes{\uP'}(F\det U)$.
De l'égalité $\uP U = \uP'$, on tire $\det(U) \uP = \uP'\,\widetilde U$
donc $\det(U) \uPdelta \subset \langle\uP'\rangle_{\delta'}$. 
On en déduit  que $\mu$ est nulle sur $\uPdelta$.

\smallskip
Pour obtenir le résultat, on peut se permettre de passer en générique pour
$\uP$ avec comme conséquence (cf. \ref{omegaresGr2}) que $\omegaRes{\uP}$ est une base
de l'espace des formes linéaires sur $\bfA[\uX]_\delta$ nulles
sur $\uPdelta$. Ainsi $\mu$ est multiple de~$\omegaRes{\uP}$.

\noindent
Le scalaire $c$ est bien déterminé car spécialisé du cas
générique pour lequel la forme $\omegaRes{\uP}$ est sans torsion.

\smallskip
iii) Prendre dans ii) $F = X^\emouton$ (le \MoutonNoir{} de $D$). 

\smallskip
iv) Prendre dans ii) $F = \nabla_\uP$.
\end {proof}

On déduit de la proposition précédente les trois résultats suivants :
\begin {coro} [En cas d'égalité $\langle\uP\rangle=\langle\uP'\rangle$] \label{ResultantEgaliteIdeal}
Soient $\uP = (P_1, \dots, P_n)$ de format $D$ 
et $\uP' = (P'_1, \dots, P'_n)$ de format~$D'$ tels que
$\langle\uP'\rangle = \langle\uP\rangle$.
Alors $\Res(\uP)$ et $\Res(\uP')$ engendrent le même idéal de $\bfA$.
\end{coro}

\begin {coro} [Triangularité] \label{omegaresTriangularite}
Soient $\uP = (P_1, \dots, P_n)$ et $\uP' = (P'_1, \dots, P'_n)$ deux systèmes
de même format $D$ vérifiant:
$$
\forall\, i, \qquad 
P'_i - P_i \in \langle P_{i+1}, \dots, P_n\rangle 
$$
Alors:
$$
\langle\uP\rangle = \langle\uP'\rangle, \qquad
\omegaRes{\uP} = \omegaRes{\uP'}, \qquad 
\Res(\uP) = \Res(\uP')
$$
\end {coro}

\begin {proof}
Par hypothèse, il y a $U \in \bbM_n(\bfA[\uX])$ triangulaire inférieure à diagonale
unité où chaque $U_{ij}$ est homogène de degré $d_j - d_i$ telle que $\uP' = \uP\,U$.
Alors $U$ est inversible donc $\langle\uP\rangle = \langle\uP'\rangle$.

\smallskip
On peut se permettre de passer en terrain générique au dessus de
$\bbZ$ pour $\uP$ et $U$, en désignant par $\bfA = \bbZ[\uT] $ un
anneau de polynômes où $\uT$ est une famille d'indéterminées allouée
aux coefficients des $P_i$ et des coefficients $(U_{\ell,j})_{\ell >
j}$ de la matrice $U$.  D'après la proposition précédente, il existe
$c, c' \in \bfA$ tels que
$$
\omegaRes{\uP'} = c\,\omegaRes{\uP}, \qquad \omegaRes{\uP} = c'\,\omegaRes{\uP'},
\qquad
\Res(\uP') = c\,\Res(\uP), \qquad \Res(\uP) = c'\,\Res(\uP')
$$
On en déduit $cc' = 1$. Les seuls inversibles de $\bfA$ étant $\pm 1$, on a
$c = c' = \pm 1$. Pour obtenir $c=1$, il suffit de spécialiser en $U := \Id_n$.
\end {proof}

\begin {coro} Etant donné un système $\uP = (P_1, \dots, P_n)$, on ne change pas la forme $\omegares$
en ajoutant à l'un des $P_i$ une combinaison $\bfA[\uX]$-linéaire des autres $P_j$ pourvu que cette
combinaison soit homogène de degré $\deg(P_i)$. Idem pour le résultant.
\end {coro}

\medskip

Ci-dessous, on envisage le cas où la matrice $U$, qui était jusqu'à
maintenant à coefficients dans $\bfA[\uX]$, est à coefficients dans
$\bfA$.  Nous gardons les notations $[P'_1, \dots, P'_n] =
[P_1, \dots, P_n]\,U$ mais de manière à ce que chaque $P'_i$ soit
homogène nous imposons, pour le format $D = (d_1, \dots, d_n)$ de
$\uP$, de vérifier $d_i = d_j$ pour tous $i,j$. Nous avons alors le résultat suivant.

\begin {prop} \leavevmode
\label {PtimesU}

Soit $\uP$ de format $D = (q,\dots,q)$, $U \in \bbM_n(\bfA)$ et le
système $\uP'$, de même format, défini par
$$
[P'_1, \dots, P'_n] = [P_1, \dots, P_n]\, U
$$
Alors:
$$
\Res(\uP') = \det(U)^{q^{n-1}}\,\Res(\uP)
$$
\end {prop}


\begin {proof} \leavevmode

Nous pouvons supposer le système $\uP$ et la matrice $U$ génériques au
dessus de $\bbZ$. L'anneau $\bfA$ des coefficients est donc un anneau
de polynômes sur $\bbZ$ dont les indéterminées sont allouées aux
coefficients des $P_i$ et de la matrice $U$. Nous allons utiliser le
localisé $\bfA' := \bfA[1/\det(U)]$ qui possède la propriété suivante:
ses seuls inversibles sont les $\pm \det(U)^m$ avec $m \in \bbZ$.
On laisse le soin au lecteur de vérifier que cette propriété est
équivalente à l'irréductibité de $\det(U)$ dans l'anneau de polynômes~$\bfA$
au dessus de $\bbZ$.

\smallskip

De manière générale, si $\bfk$ est un anneau intègre, alors le déterminant
générique d'ordre $n$ est un polynôme à $n^2$ indéterminées au dessus de
$\bfk$, qui engendre un idéal premier, a fortiori irréductible.
On peut bien sûr invoquer que ce
déterminant est le résultant de $n$ formes linéaires génériques, qu'il
engendre l'idéal d'élimination, et de ce fait la primalité de l'idéal engendré
résulte du point iii) de~\ref{GeneriqueLocaliseSature}. 
Mais est-ce bien raisonnable?

Nous invitons le lecteur à chercher une preuve directe de
l'irréductibilité. C'est par exemple l'objet de l'exercice~17, \S3,
p. 101 du chapitre 7 (Diviseurs) d'Algèbre Commutative de Bourbaki,
dont le but est de montrer que le déterminant générique au dessus d'un
anneau intègre est extrémal (terminologie retenue par Bourbaki dans
Algèbre, chap VI, \S1.13).

\medskip

En considérant $\uP$ et $\uP'$ au dessus de ce localisé $\bfA'$ et en
appliquant le résultat de divisibilité figurant dans le dernier point de
la proposition~\ref{PremierLemmeDivisibiliteJPJ}, on obtient un inversible $c$ de $\bfA'$
tel que
$$
\Res(\uP') = c\,\Res(\uP)
$$
Pour déterminer $c$, qui est de la forme $c = \varepsilon \det(U)^m$ avec
$\varepsilon = \pm 1$ et $m \in \bbZ$, on spécialise $U$ en $t\Id_n$
où~$t$ est une indéterminée. Alors $c$ se spécialise en $\varepsilon
(t^n)^m$ et on obtient donc:
$$
\Res(tP_1, \dots, tP_n) = \varepsilon t^{nm}\,\Res(P)
$$
A gauche, en utilisant l'homogénéité du résultant, notre habituel
$\widehat d_i$ étant $q^{n-1}$, on obtient l'égalité $t^{nq^{n-1}}\Res(\uP)
=\varepsilon t^{nm}\Res(\uP)$. D'où $\varepsilon = 1$ et $m= q^{n-1}$,
ce qui termine la preuve.
\end {proof}

\subsection{De $n-1$ à $n$ variables et vice versa (où $P_n=X_n$: Poisson pour bébés)}

Dans les deux résultats ci-dessous interviennent $\bfA[\uX] = \bfA[X_1, \dots, X_n]$
et $\bfA[\uX'] = \bfA[X_1, \dots, X_{n-1}]$. En plus de l'inclusion $\bfA[\uX'] \subset
\bfA[\uX]$, on dispose d'une surjection $\bfA$-linéaire $\bfA[\uX] \twoheadrightarrow \bfA[\uX']$
définie par $F \mapsto F' := F(X_n := 0)$, surjection qui est l'identité sur $\bfA[\uX']$ et vérifie
$F - F' \in \langle X_n\rangle$.

\medskip

Commençons par expliquer le titre de la section.  Poisson est celui de
la formule de Poisson qui relie $\Res(P_1, \ldots, P_n)$ et
$\Res(P'_1, \ldots, P'_{n-1})$:
$$
\hbox {Norme}\big({\widetilde {P_n}}\big) = \frac{\Res(P_1, \ldots, P_n)}
{\Res(P'_1,\ldots, P'_{n-1})^{d_n}}
\qquad\qquad
\widetilde {P_n} := P_n(X_n := 1) \in \bfA[\uX']
$$
Quel est le sens de cette formule et son lieu de vie?  Cette égalité a
lieu dans le localisé $\bfA_{R'}$ de~$\bfA$ par le scalaire $R'
:= \Res(P'_1,\ldots, P'_{n-1})$, de sorte que le dénominateur
$R'^{d_n}$ y est inversible.  Et la norme est celle de la
$\bfA_{R'}$-algèbre $\bfA_{R'}[\uX'] / \langle P'_1,\ldots,
P'_{n-1}\rangle$.  Cette dernière algèbre est un module projectif de
type fini, ce qui permet de définir la norme d'un élément comme le
déterminant de la multiplication par cet élément.  Précision: cette
algèbre est en fait stablement libre, de rang constant $d_1 \ldots
d_{n-1}$.

\smallskip

Pour plus de précisions autour de la formule de Poisson, nous 
renvoyons à l'article de Jean-Pierre Jouanolou, ``Le formalisme du résultant'',
en~\cite[sections 3-4]{J3}, les 2 sections occupant une vingtaine de pages.
On y constate qu'il faut travailler un peu pour, par exemple, montrer  
que la $\bfA_{R'}$-algèbre $\bfA_{R'}[\uX'] / \langle P'_1,\ldots,
P'_{n-1}\rangle$ est stablement libre. 

\medskip

\emph{Ici}, nous n'utilisons pas cette formule de Poisson et nous allons
nous contenter de $P_n = X_n^{d_n}$ (en commençant par $d_n = 1$!).
Mais nous tenions à signaler l'existence de cette formule générale, qui,
lorsque $P_n= X_n^{d_n}$ (pour lequel $\widetilde {P_n} = 1$) 
fournit dans $\bfA$:
$$
\Res(P_1, \ldots, P_n) = \Res(P'_1,\ldots, P'_{n-1})^{d_n}
$$
Nous commençons par le cas de bébé $P_n = X_n$ en montrant l'égalité
des résultants ci-dessus via l'étude des formes linéaires~$\omegares$ et
nous traiterons le cas $P_n = X_n^{d_n}$ dans la prochaine section Multiplicativité.
Nous avouons que notre approche est un peu laborieuse.

\begin {lem}[De $n-1$ à $n$ variables]
Etant donné un système $\uQ = (Q_1, \dots, Q_{n-1})$ de $\bfA[\uX']$
de format $D' = (d_1, \dots, d_{n-1})$, on définit le système $\uQ^+$ de
$\bfA[\uX]$ par $\uQ^+ = (Q_1, \dots, Q_{n-1}, X_n)$ de format $D =
(d_1, \dots, d_{n-1}, 1)$. Ils sont de même degré critique $\delta$.
Alors:

\begin {enumerate} [i)]
\item
La forme $\omegaRes{\uQ^+}$ prolonge $\omegaRes{\uQ}$ et vérifie
$$
\forall\, F \in \bfA[\uX]_\delta, \quad \omegaRes{\uQ^+}(F) = \omegaRes{\uQ}(F') \qquad \text{où $F' := F(X_n := 0)$}
$$
\item
On a $\Res(\uQ) = \Res(\uQ^+)$.
\end {enumerate}
\end {lem}

\begin {proof}
Nous pouvons supposer le système $\uQ$ générique sur $\bbZ$ donc $\bfA
= \bbZ[\text{indets pour les $Q_j$}]$. On sait alors
(cf \ref{omegaresGr2}) que $\omegaRes{\uQ}$ est une base de l'espace
des formes linéaires sur $\bfA[\uX']_\delta$ nulles sur
$\langle\uQ\rangle_\delta$.

\medskip
\noindent
i) Puisque $\langle\uQ\rangle_\delta \subset
\langle\uQ^+\rangle_\delta$, la restriction de $\omegaRes{\uQ^+}$
à $\bfA[\uX']_\delta$ est nulle
sur $\langle\uQ\rangle_\delta$ donc il y a une constante $c \in \bfA$ telle
que:
$$
\omegaRes{\uQ^+} = c\, \omegaRes{\uQ}  \quad \text{sur}\quad  \bfA[\uX']_\delta
$$
Pour déterminer $c$, on évalue l'égalité ci-dessus en le \MoutonNoir{}
$\uX^\emouton$ du format $D$ qui est aussi celui du format $D'$. Ces
évaluations sont des polynômes homogènes en chaque $Q_j$ de même poids
$\widehat {d}_j -1$ (cf~\ref{PoidsNormalisationMacRae}) donc
$c \in \bbZ$.  Spécialisons $\uQ$ en le jeu étalon $(X_1^{d_1}, \dots,
X_{n-1}^{d_{n-1}})$ de $D'$, de sorte que $\uQ^+$ se spécialise en le
jeu étalon de $D$. On obtient alors $c=1$ puisque les deux formes
spécialisées sont~$(X^\emouton)^\star$ (pour $\uQ^+$) et sa
restriction à $\bfA[\uX']_\delta$ (pour $\uQ$).

\smallskip
\noindent
Soit $F \in \bfA[\uX]_\delta$. Puisque $F - F' \in \langle X_n\rangle_\delta \subset \langle\uQ^+\rangle_\delta$,
on a $\omegaRes{\uQ^+}(F) = \omegaRes{\uQ^+}(F')$. Et $\omegaRes{\uQ^+}(F') = \omegaRes{\uQ}(F')$,
parce que $\omegaRes{\uQ^+}$ prolonge $\omegaRes{\uQ}$.

\medskip
\noindent
ii) Soit $\nabla_\uQ$ un déterminant bezoutien de $\uQ$. C'est aussi un déterminant
bezoutien de $\uQ^+$. En effet, on dispose d'une matrice homogène $\dsV' \in
\bbM_{n-1}(\bfA[\uX'])$ telle que $\uQ = \uX'\,\dsV'$, de déterminant $\nabla_\uQ$.
La matrice $\dsV = \diag(\dsV',1)$ est une matrice bezoutienne de $\uQ^+$ de
déterminant $\nabla_\uQ$.

\smallskip
\noindent
L'égalité $\omegaRes{\uQ}(\nabla_\uQ) = \omegaRes{\uQ^+}(\nabla_\uQ)$ fournit l'égalité
des résultants.
\end {proof}

\begin {coro} [De $n$ à $n-1$ variables quand $P_n=X_n$]
\label{WhenPnIsXn}
Soit $\uP = (P_1, \dots, P_{n-1}, X_n)$ un système de format $D = (d_1, \dots, d_{n-1}, 1)$.
On introduit les deux systèmes, le premier dans $\bfA[\uX]$ de format~$D$,
le second dans $\bfA[\uX']$ de format $D':= (d_1, \dots, d_{n-1})$:
$$
\uP^\circ = (P'_1, \dots, P'_{n-1}, X_n), \qquad\qquad
\uP' = (P'_1, \dots, P'_{n-1}) \qquad \qquad \text{où $P'_i := P_i(X_n := 0)$}
$$
\begin {enumerate} [i)]
\item
On a l'égalité des formes linéaires : $\omegaRes{\uP} = \omegaRes{\uP^\circ}$.

\item
L'application linéaire $\bfA[\uX]_\delta \to \bfA[\uX']_\delta$ définie par
$F \mapsto F'$ relie les deux formes $\omegaRes{\uP}$ et $\omegaRes{\uP'}$ 
de la manière suivante: $\omegaRes{\uP}(F) = \omegaRes{\uP'}(F')$ 
pour tout $F \in \bfA[\uX]_\delta$.

\item
Les trois résultants sont égaux:
$$
\Res(\uP) = \Res(\uP^\circ) = \Res(\uP')
$$
\end {enumerate}
\end {coro}

\begin {proof} \leavevmode

\noindent
i) Résulte du fait que $P_i - P'_i \in \langle X_n\rangle$ pour $1 \leqslant i \leqslant n-1$ et
du corollaire~\ref{omegaresTriangularite} appliqué à $\uP$ et $\uP^\circ$, corollaire
qui donne aussi $\Res(\uP) = \Res(\uP^\circ)$.

\medskip
\noindent
ii) et iii) découlent du point précédent et du lemme en prenant $\uQ = \uP'$, auquel cas $\uQ^+ = \uP^\circ$.
\end {proof}

\subsection {Multiplicativité du résultant et de la forme $\omegares$}

\begin{prop}[Multiplicativité du résultant] \label{MultiplicativiteResultant}
Soit $P_1, P'_1, P_2, \ldots, P_n$ des polynômes homogènes de degrés $d_1, d'_1, d_2, \ldots, d_n$. 
On a alors l'égalité :
$$
\Res(P_1P'_1,\, P_2, \ldots, P_n) 
\ =\  
\Res(P_1, P_2, \ldots, P_n) \ \Res(P'_1, P_2, \ldots, P_n)
$$
\end{prop}

\index{théorème!de multiplicativité!du résultant}%
%
%

\begin{proof}
Notons :
$$
\uP = (P_1, P_2, \ldots, P_n), \qquad \uP' = (P'_1, P_2, \ldots, P_n), \qquad
\uQ = (P_1P'_1, P_2, \ldots, P_n)
$$
Il suffit de se placer en terrain générique au-dessus de $\bbZ$. Considérons l'anneau de polynômes 
$\bfA = \bbZ[\uU]$ où $\uU$ est une famille d'indéterminées
allouée aux coefficients de $P_1, P'_1, P_2, \ldots, P_n$.
Ainsi les suites $\uP$ et~$\uP'$ sont génériques, a fortiori régulières. 
La suite $\uQ$ est également régulière (confer 
le $(a,b,ab)$-trick en~\ref{ababRegTrick}).

On va utiliser, pour un $\bfR$-module $E$ et deux scalaires $a,b$ où $a$ est $E$-régulier, que
la suite courte ci-dessous est exacte :
$$
\xymatrix @M=0.4pc{
0 \ar[r] & 
E/bE  \ar[r]^-{\times a} & 
E/abE \ar[r] & 
E/aE \ar[r] & 
0
}
$$
Appliquons cela à $\bfR = \bfA[\uX]$, 
$E = \bfA[\uX]/\langle P_2, \ldots, P_n\rangle$, $a = P_1$, $b = P'_1$. 
On doit s'assurer que $P_1$ est régulier sur $\bfA[\uX]/\langle P_2, \ldots, P_n\rangle$, 
ce qui est bien vrai car la suite $(P_2, \dots, P_n,P_1)$ est générique homogène 
donc régulière (cf.~\ref{SuiteGeneriqueReguliere}).
La suite exacte courte ci-dessus s'écrit :
$$
\xymatrix @M=0.4pc{
0 \ar[r] & 
\bfA[\uX]/\langle\uP'\rangle \ar[r]^-{\times P_1} & 
\bfA[\uX]/\langle\uQ\rangle \ar[r] &
\bfA[\uX]/\langle\uP\rangle \ar[r] & 
0
}
$$
%
Prenons la composante homogène de degré $d$, 
ce qui fournit la suite exacte courte de $\bfA$-modules :
$$
\xymatrix @M=0.4pc{
0 \ar[r] & 
\big(\bfA[\uX]/\langle\uP'\rangle\big)_d \ar[r]^-{\times P_1} & 
\big(\bfA[\uX]/\langle\uQ\rangle\big)_{d+d_1} \ar[r] &
\big(\bfA[\uX]/\langle\uP\rangle\big)_{d+d_1} \ar[r] & 
0
}
$$
Choisissons $d$ vérifiant l'inégalité :
$$
d \geqslant \delta_{\uP'}+1  \quad \hbox {avec} \quad \delta_{\uP'} 
\ \overset{\rm def}{=}\ 
(d'_1-1) + (d_2-1) + \cdots + (d_n-1)
$$
On a alors $d+d_1 \geqslant \delta_{\uQ}+1$ et $d+d_1 \geqslant \delta_{\uP}+1$. 
Ainsi, d'après~\ref{ResolutionQuotients}, 
chacun des trois $\bfA$-modules librement résolubles 
qui intervient ci-dessus est de caractéristique d'Euler-Poincaré nulle 
et les invariants de MacRae sont respectivement :
$$
\langle \Res(\uP')\rangle, \qquad \langle \Res(\uQ)\rangle, \qquad \langle \Res(\uP)\rangle
$$
La propriété de multiplicativité de l'invariant de MacRae (cf.~\ref{MacRaeMultiplicativity}) 
conduit à l'égalité 
$$
\Res(\uQ) 
\ =\  
u \, \Res(\uP)\, \Res(\uP')
$$
où $u \in \bbZ[\uU]$ est un inversible, donc $u = \pm 1$.
Les spécialisations $P_i := X_i^{d_i}$ et $P'_1 := X_1^{d'_1}$ 
permettent d'en déduire que $u = 1$.

Remarque : on aurait pu se dispenser de passer par $\bbZ$ en 
restant au-dessus de $\bfk$ quelconque. Pour cela, on aurait constaté que 
$u \in \bfk$ pour des raisons de poids, puis que $u = 1$ par normalisation.
\end{proof}

\begin{prop}[Multiplicativité de la forme $\omegares$] \label{omegaResMultiplicativity}
Soient $\uP = (P_1, \dots, P_n)$ un système de degré critique $\delta$ et $H$ un polynôme
homogène de degré $h$. On note $\widetilde {\uP} = (P_1, \dots, P_{n-1}, HP_n)$ le système
de degré critique $\delta + h$. Alors:
$$
\forall\, F \in \bfA[\uX]_\delta, \quad 
\omegaRes{\widetilde{\uP}}(HF) = \Res(P_1, \dots,P_{n-1},H)\ \omegaRes{\uP}(F)
$$
\end {prop}

\begin {proof} 
On a:
$$
\widetilde {\uP} = \uP\, U   \qquad \hbox {avec} \quad  U = \diag(1, \dots, 1, H)
$$
ce qui permet d'appliquer le résultat de divisibilité~\ref{PremierLemmeDivisibiliteJPJ} à $\uP' = \widetilde {\uP}$.
Comme $\det(U) = H$, celui-ci fournit:
$$
\omegaRes{\widetilde{\uP}}(HF) = c\ \omegaRes{\uP}(F)
$$
où la constante $c$ vérifie
$$
\Res(\widetilde {\uP}) = c\, \Res(\uP)
$$
Mais d'après le théorème de multiplicativité du résultant (cf. \ref{MultiplicativiteResultant}) : 
$$
\Res(\widetilde {\uP}) = \Res(P_1, \dots,P_{n-1},H)\, \Res(\uP)
$$
En rapprochant les deux expressions de $\Res(\widetilde {\uP})$, on peut se permettre,
quitte à mettre pour $\uP$ un peu de généricité dans le moteur, de simplifier
par $\Res(\uP)$. Ce qui fournit $c = \Res(P_1, \dots,P_{n-1},H)$ et termine la
preuve.
\end {proof}

\begin {coro} [De $n$ à $n-1$ variables quand $P_n=X_n^{d_n}$] \label{WhenPnIsXndn}
On note $\bfA[\uX'] = \bfA[X_1, \dots, X_{n-1}]$ et $F' = F(X_n :=
0) \in \bfA[\uX']$ pour $F \in \bfA[\uX]$.  Etant donné un système
$\uP$ de $\bfA[\uX]$ de format $D = (d_1, \dots, d_n)$, on définit le
système $\uP' = (P'_1, \dots, P'_{n-1})$ de $\bfA[\uX']$ de format $D'
= (d_1, \dots, d_{n-1})$ et on désigne par $\delta'$ son degré critique.

Alors, lorsque \framebox [1.1\width][c]{$P_n = X_n^{d_n}$}, on dispose de deux
formules de récurrence. Une pour le résultant:
$$
\Res(\uP) = \Res(\uP')^{d_n}
$$
Et une autre (partielle) pour la forme $\omegares$:
$$
\forall\, F \in \bfA[\uX]_{\delta'}, 
\qquad
\omegaRes{\uP}(X_n^{d_n-1}\,F) \ = \ 
\Res(\uP')^{d_n-1}\, \omegaRes{\uP'}(F') 
$$
En particulier:
$$
\omegaRes{\uP}(X^\emouton) \ = \ \Res(\uP')^{d_n-1}\, \omegaRes{\uP'}(X^{D' - \mathds1})
$$
\end {coro}

\begin {proof} 
Notons $\uQ = (P_1, \dots, P_{n-1}, X_n)$, qui est de degré critique $\delta'$.

\medskip

$\rhd$ Le théorème de multiplicativité du résultant fournit $\Res(\uP) = \Res(\uQ)^{d_n}$.
Or, d'après le corollaire~\ref{WhenPnIsXn}
intitulé \og de $n$ à $n-1$ variables quand $P_n=X_n$\fg{}, on a $\Res(\uQ) = \Res(\uP')$.
D'où l'égalité $\Res(\uP) = \Res(\uP')^{d_n}$.

\medskip

$\rhd$ Appliquons le théorème de multiplicativité de $\omegares$ (cf.~\ref{omegaResMultiplicativity})
à $\uQ$ et $H = X_n^{d_n-1}$ de sorte que $\widetilde {\uQ} = \uP$:
$$
\omegaRes{\uP}(X_n^{d_n-1}\,F) = \Res(P_1, \dots, P_{n-1}, X_n^{d_n-1})\,
\omegaRes{\uQ}(F)
$$
D'une part, $\Res(P_1, \dots, P_{n-1}, X_n^{d_n-1})
= \Res(\uP')^{d_n-1}$ d'après le point précédent.  D'autre part, le
même corollaire~\ref{WhenPnIsXn}-ii) donne l'égalité
$\omegaRes{\uQ}(F) = \omegaRes{\uP'}(F')$. D'où l'égalité de l'énoncé.

\smallskip
Le cas particulier s'obtient en prenant $F = X^{D'-\mathds1}$ pour lequel $HF = X^\emouton$.
\end {proof}

\subsection {Permutation et composition linéaire:
             $(P_1,\dots,P_n) \rightsquigarrow (P_1\circ u,\dots, P_n\circ u)$}

\subsubsection*{Suite permutée}
\begin{prop}[Invariants de la suite permutée] \label{SuitePermutee}
Soit $\sigma \in \fS_n$. En notant $\epsilon(\sigma)$ la signature de $\sigma$ 
et $\uP_\sigma = (P_{\sigma(1)}, \dots, P_{\sigma(n)})$, on a les égalités
$$
\omegaRes{\uP_\sigma} = \epsilon(\sigma)^{d_1\cdots d_n-1}\,\omegaRes{\uP},
\qquad
\Res(\uP_\sigma)
\ = \ 
\epsilon(\sigma)^{d_1\cdots d_n}\,
\Res(\uP)
$$
\end{prop}

\label {NOTA11-Psigma}%

\begin{proof}
Il suffit de faire la preuve en générique, \idest{} avec $\bfA
=\bfk[\indetsPi]$.  Les modules invariants de MacRae attachés aux deux
suites $\uP$ et $\uP_\sigma$ sont les mêmes puisque ces modules ne
dépendent que de $\bfA[\uX]/\langle\uP\rangle
= \bfA[\uX]/\langle\uP_\sigma\rangle$. Ainsi, il existe un inversible
$\lambda \in \bfA$ tel que $\omegaRes{\uP_\sigma}
= \lambda\,\omegaRes{\uP}$.

\smallskip
\noindent
Par ailleurs, en notant $\nabla_\uP$ un déterminant bezoutien de
$\uP$, on vérifie facilement que $\epsilon(\sigma)\, \nabla_\uP$ est un
déterminant bezoutien de $\uP_\sigma$. En évaluant l'égalité des
formes linéaires $\omegares$ en $\nabla_\uP$, on obtient l'égalité
$\epsilon(\sigma)^{-1}\,\Res(\uP_\sigma) = \lambda\,\Res(\uP)$.
Pour des raisons de poids, on a $\lambda \in \bfk$.
On peut alors spécialiser, sans affecter~$\lambda$.

Spécialisons en le jeu étalon $\uX^D = (X_1^{d_1}, \dots, X_n^{d_n})$ de résultant $1$.
On a alors 
$$ 
\epsilon(\sigma)^{-1}\, \Res(X_{\sigma(1)}^{d_{\sigma(1)}}, \dots, X_{\sigma(n)}^{d_{\sigma(n)}}) = \lambda
$$
L'utilisation répétée de la multiplicativité du résultant (cf.~\ref{MultiplicativiteResultant}) fournit :
$$
\Res\big(X_{\sigma(1)}^{d_{\sigma(1)}}, \dots, X_{\sigma(n)}^{d_{\sigma(n)}}\big)
\ =\  
\Res(X_{\sigma(1)}, \dots, X_{\sigma(n)})^{d_{\sigma(1)} \cdots {d_{\sigma(n)}}}
$$
Or le résultant de $n$ formes linéaires est leur déterminant (cf.~\ref{ResultantFormesLineaires}), 
donc $\Res(X_{\sigma(1)}, \dots, X_{\sigma(n)}) = \epsilon(\sigma)$ ; 
ce qui termine la preuve.
\end{proof}

\subsubsection*{Transformation linéaire (composition à droite)}

Pour $F \in \bfA[\uX] = \bfA[X_1, \ldots, X_n]$ et $u \in \bbM_n(\bfA)$, on définit
$F \circ u \in \bfA[\uX]$ de la manière suivante:
$$
F \circ u = F(Y_1, \ldots, Y_n)  \qquad \hbox{avec}\qquad
\begin {bmatrix} Y_1\cr \vdots\cr Y_n\cr \end {bmatrix} =
u 
\begin {bmatrix} X_1\cr \vdots\cr X_n\cr \end {bmatrix}
$$
On utilisera la notation alternative $F^u$ au lieu de $F\circ u$. Il est clair
que $F \mapsto F^u$ est un $\bfA$-endomorphisme de l'anneau $\bfA[\uX]$, gradué
de degré $0$, et que $(F^u)^v = F^{uv}$ pour $u,v \in \bbM_n(\bfA)$.

\medskip
\noindent
Soit $\uP = (P_1, \ldots, P_n)$ un système homogène de format de degrés $(d_1, \ldots,
d_n)$ et $u \in \bbM_n(\bfA)$.  
On note~$\uP^u$ le système :
$$
\uP^u
\ = \ 
(P_1^u, \ldots, P_n^u)
$$
C'est un système homogène de même format que $\uP$.

\label {NOTA11-Fou}%
\label {NOTA11-Pou}%

\medskip

\noindent
\textbf{Bezoutien de $\uP^u$}

\`A partir d'une matrice bezoutienne $\dsV$ pour $\uP$
\idest{} 
$\dsV \in \bbM_n(\bfA[\uX])$ telle que $[\uP] = [\uX] \cdot \dsV$, on
peut en construire une pour $\uP^u$, disons $\dsV'$, 
vérifiant $[\uP^u] = [\uX]\cdot \dsV'$ et 
$\det(\dsV') = \det(u)\, \det(\dsV)^u$. 

Voici comment : 
en appliquant l'endomorphisme $F \mapsto F^u$ à l'égalité
matricielle $[\uP] = [\uX]\cdot \dsV$, on obtient, en notant~$\dsV^u$ la
matrice $(\dsV_{i,j}^u)$:
$$
[\uP^u]
\  = \ 
[\uX^u]
\cdot \dsV^u 
\ =\ 
[\uX]
\cdot \transpose{u} \cdot \dsV^u
$$
et on prend $\dsV' = \transpose{u} \cdot \dsV^u$.

\medskip

On peut résumer cela par l'égalité (légèrement 
abusive car elle sous-entend que le bezoutien est unique) :
$$
\nabla_{\uP^u} \ = \  
\det(u)\, {(\nabla_{\uP})}^u
\leqno(\star)
$$

\begin {prop}\leavevmode

\noindent
En conservant les notations précédentes, on a, pour $F \in \bfA[\uX]_\delta$, l'égalité:
$$
\omegaRes{\uP^u}(F\circ u) 
\ = \ 
\det(u)^{d_1\cdots d_n - 1}\, 
\omegaRes{\uP}(F)
$$
\end {prop}

\begin {proof} \leavevmode

Il suffit de démontrer la formule dans le cadre générique (pour $\uP$
et $u$) suivant: $u$ est la matrice générique $u = (u_{ij})$ où
$u_{ij}$ sont $n^2$ indéterminées sur $\bbZ$, l'anneau $\bfk$ est
$\bbZ[(u_{ij})_{1 \leqslant i,j \leqslant n}]$ et $\bfA = \bfk[\indetsPi]$.  

On 
rappelle (cf.~\ref{PoidsNormalisationMacRae}) 
qu'en posant $\widehat {d_i} = \prod_{j\ne i}d_j$, chaque
coefficient de la forme linéaire $\omegaRes{\uP}$ est homogène en
$P_i$ de poids $\widehat {d_i}-1$; il en est de même pour les formes
linéaires $\omegaRes{\uP^u}$ et 
$F \mapsto \omegaRes{\uP}(F^u)$.

\medskip

Pour $F \in \langle\uP\rangle_\delta$, on a $F^u \in \langle\uP^u\rangle_\delta$,
par conséquent la forme linéaire $F \mapsto \omegaRes{\uP^u}(F^u)$
est nulle sur~$\uPdelta$.

\medskip
Puisque $\uP$ est générique, la forme linéaire $\omegaRes{\uP}$ est un
générateur du $\bfA$-module des formes linéaires sur
$\bfA[\uX]_\delta$ nulles sur $\langle\uP\rangle_\delta$
(cf.~\ref{omegaresGr2}-ii).
Il y a donc
$c \in \bfA = \bfk[\indetsPi]$ tel que:
$$
\forall\, F \in \bfA[\uX]_\delta, \qquad 
\omegaRes{\uP^u}(F^u) 
\ = \ 
c\  \omegaRes{\uP}(F) 
$$
En utilisant les propriétés de poids, 
on voit que le multiplicateur $c$ est homogène en chaque $P_i$ de
poids~$0$ donc $c \in \bfk$. 

En multipliant par $\det u \in \bfA$ l'égalité précédente, 
on obtient 
$$
\forall\, F \in \bfA[\uX]_\delta, \qquad
\omegaRes{\uP^u}(\det(u)\, F^u) 
\ = \ 
c\det(u)\ \omegaRes{\uP}(F)
$$
Puis en évaluant en $F = \nabla_{\uP}$ et en utilisant $(\star)$, on obtient 
$$
\Res(\uP^u)
\ = \ 
c\det(u)\ 
\Res(\uP)
$$
Pour achever la détermination de $c$
(indépendant de $\uP$), on considère l'égalité précédente 
avec le jeu étalon $\uP = \uX^D$.
On a donc 
$$
\Res\big((X_1^u)^{d_1}, \dots, (X_n^u)^{d_n}\big) 
\ = \ 
c\det(u) 
\underbrace{\Res(\uX^D)}_{=1}
$$
En utilisant d'une part la multiplicativité du résultant et d'autre part
le fait que le résultant de $n$ formes linéaires est leur déterminant, on obtient:
$$
\underbrace{%
\Res(X_1^u, \dots, X_n^u)^{d_1\cdots d_n}}_{=\det(u)^{d_1\cdots d_n}}
\ = \ 
c \det u
$$
d'où la valeur $c = \det(u)^{d_1\cdots d_n-1}$.
\end {proof}

\begin {coro} \label {LinearComposition} \leavevmode 
\noindent
Dans le contexte de la proposition précédente, on a:
$$
\Res(P_1 \circ u, \dots, P_n \circ u)
\ = \ 
\det(u)^{d_1\cdots d_n}\, \Res(P_1, \dots, P_n)
$$
\end {coro}

\begin {proof}
La proposition précédente fournit l'égalité
$$
\omegaRes{\uP^u}(\det(u)\,F^u) 
\ = \ 
\det(u)^{d_1\cdots d_n}\, \omegaRes{\uP}(F)
$$
que l'on applique à $F = \nabla_{\uP}$.
Comme $\det(u)\,(\nabla_{\uP})^u = \nabla_{\uP^u}$, 
on obtient l'égalité requise !
\end {proof}

\subsubsection*{Un zeste de discriminant}

Etant donné un polynôme homogène $F \in \bfA[\uX]$, de degré $d\ge 2$,
en notant $\partial_i F$ sa dérivée partielle en~$X_i$, qui est un
polynôme homogène de degré $d-1$, on définit son r-discriminant
comme le résultant de format $(d-1, \dots, d-1)$:
$$
\Discr(F) = \Res(\partial_1 F, \dots, \partial_n F)
$$
\label{NOTA11-Discr}%
La lettre r est un acronyme pour ``résultant''. Ce r-discriminant
n'est pas le discriminant réel de~$F$. Ce dernier, dit parfois
discriminant divisé, s'obtient en divisant $\Discr(F)$ par une
puissance adéquate de~$d$ de manière à obtenir, en générique au dessus
de $\bbZ$, un élément primitif, cf.  l'introduction de Demazure
en \cite{Demazure2} et les sections 5,6 de cet article. A cet effet,
on introduit le polynôme à coefficients entiers:
$$
a_n(T) = \dfrac{(T-1)^n - (-1)^n}{T}
\qquad \text{et on pose} \qquad
a(n,d) = a_n(d) = \dfrac{(d-1)^n - (-1)^n}{d}
$$
Pour $n,d \ge 1$, l'entier $a(n,d)$ est $\ge 0$.
Le discriminant de $F$ est le quotient exact:
$$
\Disc(F) = \dfrac{\Discr(F)}{d^{a(n,d)}}
$$
\label{NOTA11-Disc}%
Par exemple, pour $d = 2$, on a $a(n,2) = 0$ pour $n$ pair et $a(n,2)
= 1$ pour $n$ impair.  Ainsi pour une forme quadratique $F$,
le r-discriminant $\Discr(F)$ est le déterminant des $n$ dérivées partielles, à diviser
par $2$ lorsque $n$ est impair pour obtenir $\Disc(F)$. On a donc
$$
\Disc(aX^2 + bXY + cY^2) = \begin {vmatrix} 2a & b\\ b & 2c \\ \end {vmatrix} =
4ac - b^2
$$
tandis que pour
$$
F = aX^2 + bXY + cY^2 + dXZ + eYZ + fZ^2
$$
on a
$$
\Disc(F) = \dfrac {1}{2}
\begin {vmatrix} 2a & b &d\\ b & 2c &e \\ d &e &2f\end {vmatrix} =
4acf + bde - ae^2 - cd^2 - fb^2
$$
La divisibilité par 2 lorsque $n$ est impair peut se justifier comme
suit.  Soit $A \in \bbM_n(\bfR)$ une matrice \emph {symétrique} telle
que chaque coefficient diagonal $a_{ii} \in 2\bfR$; alors $\det(A) \in
2\bfR$. Pour le voir, on considère le développement habituel sur
$\fS_n$, $\det(A) = \sum_\sigma a_\sigma$ où $a_\sigma
= \varepsilon(\sigma)a_{1\sigma(1)} \cdots a_{n\sigma(n)}$.  Si
$\sigma^2 = \Id_n$, alors $\sigma$ possède un point fixe (car $n$ est
impair) donc $a_\sigma \in 2\bfR$; sinon, $\sigma \ne \sigma^{-1}$,
auquel cas $a_\sigma = a_{\sigma^{-1}}$ (car $A$ est symétrique) et
donc $a_\sigma + a_{\sigma^{-1}} \in 2\bfR$.

\smallskip

Pour une étude détaillée du discriminant, cf. l'article de Demazure et
l'article \cite[section 4]{BuseJouanolouDiscriminantII} de
Busé~\& Jouanolou. En guise d'application des règles de transformation
du résultant, nous nous contentons ici de montrer le résultat suivant
(cf. le point e) de la proposition~5.11 de \cite{Demazure2} ou la
proposition~4.13 de \cite{BuseJouanolouDiscriminantII}).

\index{discriminant d'un polynôme homogène}%

\begin {prop} \leavevmode

Soit $F \in \bfA[\uX]$ un polynôme homogène de degré $d \ge 2$ et
$U \in \bbM_n(\bfA)$. Alors, pour le polynôme homogène $F\circ U$ de
même degré $d$, on dispose de la formule:
$$
\Discr(F\circ U) = \det(U)^{d(d-1)^{n-1}}\,\Discr(F)
$$
\end {prop}

\begin {proof} \leavevmode

Posons $G = F\circ U$ i.e. $G(\uX) = F(\uY)$ avec $Y_i = \sum_j u_{i,j} X_j$. 
Calculons les dérivées partielles de $G$, par exemple la première et, pour
y voir quelque chose, prenons provisoirement $n=3$
$$
G = F(u_{11}X_1 + u_{12}X_2 + u_{13}X_3,\ u_{21}X_1 + u_{22}X_2 + u_{23}X_3,\
u_{31}X_1 + u_{32}X_2 + u_{33}X_3)
$$
D'où
$$
\partial_1 G = u_{11} (\partial_1 F)(\uY) + u_{21} (\partial_2 F)(\uY) + u_{31} (\partial_3 F)(\uY)
$$
On voit ainsi qu'il faut continuer à utiliser \og les notations de vecteurs couchés\fg{}, ce qui
permet d'écrire:
$$
[\partial_1 G, \dots, \partial_n G] = \big[(\partial_1 F) \circ U, \dots, (\partial_n F)\circ U\big]\, U
$$
La proposition \ref{PtimesU} fournit:
$$
\Res(\partial_1 G, \dots, \partial_n G) = \det(U)^{(d-1)^{n-1}}\,
\Res\big( (\partial_1 F) \circ U, \dots, (\partial_n F)\circ U\big)
$$
Et le corollaire de composition linéaire~\ref{LinearComposition} conduit à
$$
\Res(\partial_1 G, \dots, \partial_n G) = \det(U)^{(d-1)^{n-1}}\, \det(U)^{(d-1)^{n}}\,
\Res(\partial_1 F, \dots, \partial_n F)
$$
c'est-à-dire au résultat escompté puisque $(d-1)^{n-1} + (d-1)^n = d(d-1)^{n-1}$.
\end {proof}

\subsection{Une suite régulière particulière: le retour de $(aY,bX,cZ^3)$}
\label{SuiteTypiqueYXZ3}

On veut illustrer ici plusieurs phénomènes: d'une part, l'importance de
la clause \og couvrir le jeu étalon généralisé\fg{} dans l'obtention du
résultant $\Res(\uP)$ comme pgcd fort; et d'autre part une méthode que nous qualifions
de \og générisation \fg{} de Demazure, consistant à ajouter à $\uP$
un multiple générique du jeu étalon.

\medskip

\index{système!z@$(aY,bX,cZ^3)$}%
Prenons l'exemple typique du jeu $\uP = (aY, bX, cZ^3)$
pour lequel $\Res(\uP) = -a^3b^3c$ (cette dernière égalité peut être prouvée 
avec le cas d'école \ref{Resultant11d} ou 
en utilisant la suite permutée, cf.~\ref{SuitePermutee}). 
Cette suite $(aY, bX, cZ^3)$ ne couvre pas le jeu étalon généralisé,
donc on ne peut pas appliquer~\ref{ProprietesBasiquesResultant}: il n'y~a
aucune raison pour que $\Res(\uP)$ soit \emph{un} pgcd fort des
$\big(\omega^\sigma(\nabla)\big)_{\sigma\in\Sigma_2}$ ou des
$\big(\det W_{1,\delta+1}^\sigma\big)_{\sigma \in \Sigma_2}$.  Et
c'est effectivement le cas, puisque d'après~\ref{YXZ3}, on a
$\omega^\sigma(\nabla) = 0$ et $\det W_{1,d}^\sigma = 0$ pour tout
$d \geqslant \delta+1$ et $\sigma \in \fS_3$.  Bien entendu,
$-a^3b^3c$ n'est pas le pgcd fort d'une famille constituée de 0,
famille qui d'ailleurs n'admet pas de pgcd fort en notre sens.

\medskip

Voyons ce que fournit la méthode de \og générisation \fg{} de Demazure.
Celle-ci consiste à ajouter à $\uP$ un multiple générique du jeu étalon,
autrement dit, il s'agit de considérer $\uP^{[t]} := \uP + t\uX^D$ où $t$ 
est une indéterminée sur~$\bfA$. 
On revient ensuite à la suite initiale avec la spécialisation $t := 0$.
Notons que cette suite $\uP^{[t]}$ n'est pas une suite 
qui couvre le jeu étalon généralisé en notre sens (cf.~\ref{JeuCouvrantEtalon}).

\medskip

Pour $\sigma$ variant dans $\Sigma_2 = \{\Id_3, (2,3),
(1,3)\} \subset \fS_3$, examinons les différents déterminants $\det W_{1,d}^\sigma(\uP^{[t]})$
avec $d = \delta+1$ et $\uP^{[t]} =
(aY + tX,\, bX + tY,\, cZ^3 + tZ^3)$.  Pour $\sigma = \Id_3$:
$$
W_{1,\delta+1}(\uP^{[t]}) \ = \ 
\EastBordermatrix{
t & . & . & . & . & . & . & . & . & . & \Heti{X^{3}} \\ 
a & t & . & . & . & . & . & . & . & . & \Heti{X^{2}Y} \\ 
. & . & t & . & . & . & . & . & . & . & \Heti{X^{2}Z} \\ 
. & a & . & t & . & . & b & . & . & . & \Heti{XY^{2}} \\ 
. & . & a & . & t & . & . & b & . & . & \Heti{XYZ} \\ 
. & . & . & . & . & t & . & . & b & . & \Heti{XZ^{2}} \\ 
. & . & . & a & . & . & t & . & . & . & \Heti{Y^{3}} \\ 
. & . & . & . & a & . & . & t & . & . & \Heti{Y^{2}Z} \\ 
. & . & . & . & . & a & . & . & t & . & \Heti{YZ^{2}} \\ 
. & . & . & . & . & . & . & . & . & t+c& \Heti{Z^{3}} \\ 
}
$$
Cette matrice présente une structure triangulaire en trois blocs d'ordre $3$, $6$, $1$, qui conduit à:
$$
\det W_{1,\delta+1}(\uP^{[t]})\, = \,
t^3 (t^2 - ab)^3 (t + c)
$$
Pour $\sigma$ égale à la transposition $(2,3)$, on a 
$$
W_{1,\delta+1}^\sigma(\uP^{[t]}) \ = \ 
W_{1,\delta+1}(\uP^{[t]}) 
$$
Cette égalité s'explique par le fait que l'ordre $<_\sigma$ sur
$\{1,2,3\}$ est $1 < 3 < 2$ et que tout ensemble $E$ de $D$-divisibilité
pour $D = (1,1,3)$ en degré $\delta+1 = 3$ ne contient pas $\{2,3\}$
donc $\min_{<_\sigma}(E) = \min(E)$: pas de différence entre l'ordre
habituel et l'ordre $<_\sigma$. Enfin pour $\sigma$ égale à la
transposition $(1,3)$, on a :
$$
W_{1,\delta+1}^\sigma(\uP^{[t]}) \ = \ 
\EastBordermatrix{
t & b & . & . & . & . & . & . & . & . & \Heti{X^{3}} \\ 
a & t & . & b & . & . & . & . & . & . & \Heti{X^{2}Y} \\ 
. & . & t & . & b & . & . & . & . & . & \Heti{X^{2}Z} \\ 
. & . & . & t & . & . & b & . & . & . & \Heti{XY^{2}} \\ 
. & . & a & . & t & . & . & b & . & . & \Heti{XYZ} \\ 
. & . & . & . & . & t & . & . & b & . & \Heti{XZ^{2}} \\ 
. & . & . & . & . & . & t & . & . & . & \Heti{Y^{3}} \\ 
. & . & . & . & . & . & . & t & . & . & \Heti{Y^{2}Z} \\ 
. & . & . & . & . & a & . & . & t & . & \Heti{YZ^{2}} \\ 
. & . & . & . & . & . & . & . & . & t+c & \Heti{Z^{3}} \\ 
}
$$
Cette matrice présente une structure triangulaire en trois blocs d'ordre $5$, $4$, $1$, 
conduisant au déterminant:
$$
t(t^2 - ab)^2 \times t^2(t^2 - ab) \times (t+c)
$$
\idest{} le même que celui de $W_{1,\delta+1}(\uP^{[t]})$.

Ainsi, pour $\sigma$ variant dans $\Sigma_2$, 
le pgcd des déterminants $\det W_{1,\delta+1}^\sigma(\uP^{[t]})$
est très simple à calculer puisque ces déterminants sont égaux.
Mais ce pgcd ne fournit pas ici le résultant de $\uP^{[t]}$.

\medskip

En revanche, la générisation exposée par Demazure peut s'appliquer
puisqu'elle utilise le fait que pour toute matrice carrée
$A \in \bbM_m(\bfR)$, le déterminant $\det(t \, \Id_m + A)$ est un polynôme
unitaire en $t$ de degré~$m$, a fortiori un élément régulier de
$\bfR[t]$. Il est donc possible dans la formule de Macaulay (exposée
en section~\ref{sousSectionMacaulayRecurrence})
de déterminer le quotient (exact) via une division par un polynôme unitaire. Ici
$$
W_{2,\delta+1}(\uP^{[t]})
\ = \ 
\EastBordermatrix{
t & . & . & \Heti{X^{2}Y} \\ 
a & t & . & \Heti{XY^{2}} \\ 
. & . & t & \Heti{XYZ} \\ 
}
\qquad 
\text{d'où } \qquad 
\det W_{2,\delta+1}(\uP^{[t]}) \, = \,  
t^{3}
$$
Ainsi, le quotient exact $\dfrac{\det W_{1,\delta+1}}{\det W_{2,\delta+1}}$ pour cette suite 
$\uP^{[t]}$ vaut 
$$
(t^2 - ab)^3 (t+c)
$$
Pour la suite $\uP$, il vaut donc $-a^{3}b^{3}c$ (spécialiser en $t := 0$).
Comme annoncé au début, on trouve $\Res(\uP) = -a^3b^3c$.

\cleardoublepage

\section{Résultant et idéal d'élimination (MacRae versus annulateur)}
\label{ChapElimination}

Résumons l'organisation générale des trois idéaux intervenus
précédemment dans l'étude d'un module~$M'$ de présentation finie. 
Commençons par  la double inclusion classique entre $0$-Fitting et annulateur 
(pour plus de précisions, cf.~\ref{0FittingAnnulateur})
$$
\calF_0(M') \ \subset \  \Ann(M')\  \subset\ \sqrt{\calF_0(M')}
$$
En particulier, $\Ann(M')$ et $\calF_0(M')$ ont même racine.

\smallskip
Supposons de surcroît que $M'$ soit un module de MacRae de rang $0$. 
Alors, par définition, l'idéal invariant de MacRae
$\MacRae(M')$ est engendré par un élément régulier~$g$, 
pgcd fort de l'idéal $\calF_0(M')$, au sens où l'idéal co-facteur $\fb
= \calF_0(M')/g$ vérifie $\Gr(\fb) \geqslant 2$.

\medskip

Sur certains anneaux $\bfA$ (par exemple les anneaux principaux ou les
anneaux de Dedekind), la condition $\Gr(\fb) \geqslant 2$ entraîne
$1 \in \fb$, auquel cas $\calF_0(M') = \langle g\rangle$, a fortiori
$g \in \Ann(M')$. Mais ces anneaux ne sont absolument pas ceux qui
interviennent dans le cadre de notre étude: il faut plutôt penser aux
anneaux de polynômes à plusieurs indéterminées.  Et en général
l'invariant de MacRae n'a aucune raison d'appartenir à l'annulateur.
En voici un exemple à la fois élémentaire et générique.  Prenons une
suite $\ua = (a_1, \dots, a_n)$ de profondeur $\geqslant 2$ (par
exemple, une suite régulière de longueur $n \geqslant 2$) ; alors $M'
= \bfA/\langle \ua \rangle$ est de MacRae de rang $0$ (nous avons
donné cet exemple~page~\pageref{ExempleAsurbGr2}) et on a les deux
égalités :
$$
\Ann(M') = \langle \ua \rangle
\qquad  \qquad 
\MacRae(M') = \langle 1 \rangle
$$
On laisse au lecteur la possibilité de choisir sa suite 
régulière non unimodulaire préférée, en vue d'obtenir un exemple 
de module dont l'invariant de MacRae n'est pas dans l'annulateur.
\medskip

Cependant, dans le cadre de l'élimination, nous allons voir, pour une suite
régulière $\uP$, que le $\bfA$-module de MacRae de rang~0
$$
\bfB'_\delta = (\bfA[\uX]/\langle \uP,\nabla\rangle)_\delta \qquad\qquad
\hbox {($\nabla$ bezoutien de $\uP$)}
$$
possède une propriété spécifique: son invariant de MacRae $\Res(\uP)$,
égal à $\omegares(\nabla)$, appartient à son annulateur. Cette appartenance
peut être mise en relief en écrivant $\MacRae(\bfB'_\delta)
\in \Ann(\bfB'_\delta)$. Ainsi, dans ce cadre, 
la propriété dite de neutralisation des idéaux de profondeur $\geqslant 2$
du lemme \ref{Gr2StrongGcd},  porte bien son nom,
car tout se passe comme si 1 appartenait à l'idéal cofacteur bien
que ce ne soit absolument pas~\mbox{le cas}.

\subsection {Appartenance du résultant à l'idéal d'élimination}

Ce qui suit repose essentiellement sur la forme linéaire $\omegares = \omegaRes{\uP}
: \bfA[\uX]_\delta \to \bfA$ et sur le fait de
voir le résultant de $\uP$ comme son évaluation en $\nabla$:
$$
\Res(\uP) \ = \  \omegares(\nabla) 
$$

\begin{theo}[Précisions sur l'appartenance du résultant à l'idéal d'élimination]
\label{omegaresNablaInElimIdeal} \leavevmode

Pour un système $\uP$ quelconque, on a $\omegares(\nabla) \in \Ann(\bfB'_\delta)$,
ce qui se traduit dans $\bfA[\uX]$ par:
$$
\omegares(\nabla)\,\bfA[\uX]_\delta \ \subset \ \langle\uP,\nabla\rangle_\delta
$$
En conséquence, on a l'inclusion d'idéaux:
$$
\omegares(\nabla)\,\langle X_1, \cdots, X_n\rangle^{\delta+1} \subset \langle\uP\rangle
$$
A fortiori, le résultant de $\uP$, perçu comme $\omegares(\nabla)$, appartient à l'idéal d'élimination:
$$
\Res(\uP) \in \ElimIdeal
$$
\end{theo}

\begin{proof}

La forme linéaire $\omegares$ possède la propriété
exceptionnelle d'être de Cramer (cf.~\ref{omegaresCramer}),  à fortiori de Cramer en $\nabla$:
$$
\omegares(\nabla)F  - \omegares(F)\nabla  \in \uPdelta \qquad \forall F \in \bfA[\uX]_\delta
$$
On en déduit $\omegares(\nabla)\,\bfA[\uX]_\delta \ \subset \ \langle\uP,\nabla\rangle_\delta$,
qui n'est autre que $\omegares(\nabla) \in \Ann(\bfB'_\delta)$.

\medskip
L'inclusion d'idéaux provient de $X_i\nabla \in \langle\uP\rangle$, appartenance due à la définition même
d'un déterminant bezoutien $\nabla$ de~$\uP$ (cf~\ref{NablaDansLeSature}).
\end{proof}

\subsubsection {Analyse du rôle de $\omegares$ dans la preuve du théorème précédent}

Dans le contexte $\uP$ régulière, nous voulons maintenant insister sur le fait que
l'appartenance du résultant à l'idéal d'élimination est \emph {équivalente} à
la propriété de Cramer en $\nabla$ pour~$\omegares$.
Nous avons imposé $\uP$ régulière de
manière à ce que les $\bfA$-modules $\bfB_\delta$ et~$\bfB'_\delta$
associés à $\uP$ soient de MacRae. Nous savons alors que
$\Ann(\bfB'_\delta) = \ElimIdeal$ d'après la partie non triviale de~\ref{AnnEqualities}
(non triviale car elle utilise en arrière plan le théorème de Wiebe).

\smallskip

C'est pour mieux comprendre cette propriété spécifique
$\MacRae(\bfB'_\delta) \in \Ann(\bfB'_\delta)$ (\og le résultant est
dans l'idéal d'élimination\fg), que nous avons tenu, depuis plusieurs
chapitres, à abstraire la situation en adoptant un cadre plus général
relevant de l'algèbre commutative.  Nous venons de voir que c'est la
forme linéaire $\uomegares : \bfB_\delta
\to \bfA$, invariant de MacRae du module $\bfB_\delta$ de rang 1, qui
apporte la solution et bien entendu ce qui la relie 
au $\bfA$-module $\bfB'_\delta$ de rang~0 via $\MacRae(\bfB'_\delta)
= \uomegares(\overline\nabla)$.

\medskip

Reprenons les résultats prouvés
en~\ref{Rank1to0} et~\ref{Rank1to0MacRae} en choisissant ici de
considérer l'invariant de MacRae comme un élément défini à un
inversible près (un scalaire ou une forme linéaire) et non pas comme
l'idéal ou le module qu'il engendre.

\begin{quote} \label{RappelCramerAnn}
\it 
\`A un module $M$ de MacRae de rang $1$ et un $m_0 \in M$ sans torsion, 
on associe le module $M' := M/\bfA m_0$ de MacRae de rang $0$ 
ainsi que l'évaluation en $m_0$ 
$$
\eval_{m_0} : \begin{array}[t]{rcl}
M^\star & \longrightarrow & \bfA \\
\alpha & \longmapsto & \alpha(m_0)
\end{array}
$$
Cette application établit un double isomorphisme schématisé 
de manière verticale ci-dessous :
$$
\xymatrix @C=0.2pc{
\FittVect_1(M) = \MacRaeVect(M) \fb \ar[d]_{\eval_{m_0}}^{\simeq} & \subset &
      \Cramer_{m_0}(M) \ar[d]^{\eval_{m_0}}_{\simeq} \\ 
\calF_0(M') = \MacRae(M') \fb & \subset & \Ann(M') \\ 
}
$$
Dans ce schéma, l'invariant de MacRae $\MacRaeVect(M)$ est
une forme linéaire sur $M$ notée $\vartheta$, et
l'invariant de MacRae $\MacRae(M')$ est un scalaire, égal à
$\vartheta(m_0)$; quant à l'idéal cofacteur $\fb$ de type
fini, il vérifie $\Gr(\fb) \geqslant 2$.

\medskip 
Par conséquent, on a l'équivalence:
$$
\MacRaeVect(M) \in \Cramer_{m_0}(M) \  \iff \
\MacRae(M') \ \in \ \Ann(M')
$$
\end{quote}

Ici, la correspondance est la suivante:
$$
M = \bfB_\delta = \bfA[\uX]_\delta /\langle \uP \rangle_\delta, \qquad 
m_0 = \overline \nabla, \qquad 
M' = \bfB'_\delta = \bfA[\uX]_\delta /\langle \uP,\, \nabla \rangle_\delta
$$
La forme linéaire $\vartheta$ est la forme linéaire $\uomegares$ et 
$\vartheta(m_0) = \omegares(\nabla)$ est l'invariant de MacRae de $\bfB'_\delta$.

\medskip

Quelle est, dans ce contexte, la signification de l'équivalence finale
du résultat général ci-dessus?
A gauche, à savoir la propriété \og de Cramer en $\overline\nabla$\fg{}
de $\MacRaeVect(\bfB_\delta) = \uomegares$
$$
\uomegares(\overline\nabla)b = \uomegares(b)\overline\nabla  \quad \hbox {sur $\bfB_\delta$}
$$
est bien sûr identique à la propriété \og $\omegares$ est de Cramer en $\nabla$\fg.
Et le bilan de l'équivalence est le suivant: \emph {la propriété de Cramer en $\nabla$ de~$\omegares$
est équivalente à l'appartenance du résultant à l'idéal d'élimination}.

\medskip

Note: si on examine d'un peu plus près l'hypothèse à droite de l'équivalence
i.e. l'appartenance de $\MacRae(\bfB'_\delta) =\omegares(\nabla)$ à $\Ann(\bfB'_\delta)$, 
on voit qu'elle se traduit par:
$$
\omegares(\nabla)\bfB_\delta \subset \bfA\overline\nabla
\qquad \hbox {ou encore par} \qquad
\omegares(\nabla)\,\bfA[\uX]_\delta \ \subset\ \langle\uP,\nabla\rangle_\delta
$$
qui n'est autre, bien entendu, que l'inclusion intervenue dans la preuve du théorème~\ref{omegaresNablaInElimIdeal}.
A priori, elle ne semble pas équivalente à la propriété de Cramer en $\nabla$ de $\omegares$ et
pourtant elle l'est.

\subsection {Idéaux égaux à radical près: variations autour de $\sqrt{\protect\ElimIdeal} = \sqrt{\Res(\protect\uP)}$}

\subsubsection {Inversibilité et nullité du résultant}

Le résultat suivant utilise le théorème précédent et la vision de
$\Res(\uP)$ en degré $\delta$ et degré $d \ge \delta + 1$ via
l'invariant de MacRae $\calR_d$ de $\bfB_d$.

\begin {coro}
Pour un système $\uP$ de format $(d_1,\cdots,d_n)$ de degré critique
$\delta$, on dispose des équivalences suivantes:

\begin {enumerate}[\rm i)]  
\item
Le scalaire $\Res(\uP)$ est inversible.

\item
On a l'inclusion $\langle X_1, \cdots, X_n\rangle^{\delta+1} \subset \langle\uP\rangle$.

\item
Il existe $e$ tel que $\langle X_1, \cdots, X_n\rangle^e \subset \langle\uP\rangle$.

\item
Il existe $r$ tel que $\langle X_1^r, \cdots, X_n^r\rangle \subset \langle\uP\rangle$.

\item
Il existe $d$ tel que l'application de Sylvester $\Syl_d(\uP) :
\rmK_{1,d} := \bigoplus_{i=1}^n \bfA[\uX]_{d-d_i} e_i \to \rmK_{0,d}
:= \bfA[\uX]_d$ soit surjective. Dit autrement: $\langle\uP\rangle_d =
\bfA[\uX]_d$.
\end {enumerate}

Quelques précisions. Tout d'abord, dans ce cas, la suite $\uP$ est régulière.
Dans le point iii) si l'inclusion est vérifiée pour un $e_0$, elle l'est pour
tout $e \ge e_0$. Observations analogues pour les points suivants.
\end {coro}  

\begin {proof} \leavevmode

Si on s'appuie sur n'importe quel point autre que que le i), on voit que
l'idéal $\langle\uP\rangle$ contient une suite régulière de
longueur~$n$, à savoir une suite du type $(X_1^r, \cdots, X_n^r)$. En conséquence,
la suite (homogène)~$\uP$, de longueur $n$, est régulière.

\medskip
D'après le théorème~\ref{omegaresNablaInElimIdeal}, on a:
$$  
\Res(\uP) . \langle X_1, \cdots, X_n\rangle^{\delta+1} \subset \langle\uP\rangle
$$
En conséquence $\rm i) \Rightarrow \rm ii)$. Il est clair que $\rm ii)
\Rightarrow \rm iii)$ et $\rm iii) \iff \rm iv)$. On rappelle
à ce propos que pour tout entier $r$, il existe un entier $e$ tel que:
$$
\langle X_1, \cdots, X_n\rangle^e \subset \langle X_1^r, \cdots, X_n^r\rangle
$$
On a $\rm iii) \iff \rm v)$ avec $d=e$. Par exemple, pour passer de l'inclusion des idéaux de
$\rm iii)$ à $\rm v)$, considérer la composante homogène de degré $e$ des idéaux.

\smallskip
Enfin, supposons le point v) vérifié avec un $d$ que l'on peut
supposer $\ge \delta+1$.  L'application $\Syl_d(\uP)$ étant
surjective, son idéal déterminantiel d'ordre $s_d:=\dim\bfA[\uX]_d$
contient 1. Donc 1 est un pgcd fort de cet idéal, de sorte qu'au niveau
de l'invariant de MacRae, on a $\langle\calR_d\rangle = \langle 1\rangle$
donc $\Res(\uP)$, égal à $\calR_d$, est inversible.
Nous venons en fait d'écrire que $\bfB_d = 0$ est
d'invariant de MacRae $\MacRae(\bfB_d) = \langle\calR_d\rangle = \langle 1\rangle$!

\smallskip
On pourrait également partir de l'inclusion $\rm iv)$ et invoquer le
théorème de divisibilité \ref {PremierLemmeDivisibiliteJPJ}
$$
\Res(\uP) \mid \Res(X_1^r, \cdots, X_n^r) = 1
$$
pour en déduire $\Res(\uP)$ inversible. Mais est-ce que cela serait bien
raisonnable?
\end {proof}

\begin {coro} \label {ZeroCommunNulliteRes}
Soit $\uP$ un système possédant un zéro commun \og solide\fg{} $\uxi = (\xi_1, \cdots, \xi_n)$
dans un sur-anneau de l'anneau $\bfA$ des coefficients de $\uP$. Solide signifie ici
que pour $a \in \bfA$, on a l'implication $a\,\uxi = 0 \Rightarrow a=0$.
Alors $\Res(\uP) = 0$.
\end {coro}

\begin {proof}
Il y a un exposant $e$ tel que $\Res(\uP) \langle\uX\rangle^e \subset \langle\uP\rangle$.
En réalisant $X_i := \xi_i$, il vient $\Res(\uP) \langle\uxi\rangle^e = 0$. Si $e\ge 1$, on en
déduit $\Res(\uP) \langle\uxi\rangle^{e-1} = 0$ et de proche en proche $\Res(\uP) = 0$.
\end {proof}

\subsubsection {Résultant, idéaux de Fitting et idéal d'élimination : à radical près}

Pour un système quelconque $\uP$, nous allons maintenant compléter le résultat
$\omegares(\nabla) \in \Ann(\bfB'_\delta)$ du théorème
précédent~\ref{omegaresNablaInElimIdeal} en fournissant d'une part
un analogue en degré $d \ge \delta + 1$ et d'autre part en faisant intervenir les
idéaux de Fitting $\calF_0(\bfB'_\delta)$ et $\calF_0(\bfB_d)$.

Nous rappelons à cette occasion (cf la
définition~\ref{DefResultant} ainsi
que le chapitre~\ref{ChapMacRaeForP}, en particulier à la
page~\pageref{ResumeSpecialisationCasGenerique}) que le résultant
$\Res(\uP)$ est multiplement défini par
$$
\Res(\uP) = \omegares(\nabla) = \calR_d
$$
et que le scalaire $\calR_d = \calR_d(\uP)$ divise $\calF_0(\bfB_d)$,
propriété analogue au fait que $\omegares(\nabla)$ divise
$\calF_0(\bfB'_\delta)$. Quant aux divers idéaux de Fitting qui interviennent,
ce sont les idéaux déterminantiels suivants de l'application de Sylvester de $\uP$:
$$
\calF_0(\bfB'_\delta) = \DVect_{s_\delta}(\Syl_\delta)\text{-évalué-en-$\nabla$},
\qquad\qquad
\calF_0(\bfB_d) = \calD_{s_d}(\Syl_d)
$$
où $s_\delta = \dim \bfA[\uX]_\delta -1$ et $s_d = \dim \bfA[\uX]_d$
(ne pas oublier que $d \ge \delta+1$).  En ce qui concerne l'égalité
de gauche, $\DVect_{s_\delta}(\Syl_\delta)$ désigne le sous-module
de $\bfA[\uX]^\star$ des formes déterminantales, cf le chapitre~\ref{ChapFittingVectoriel},
page~\pageref{Fitting1toFitting0Elimination}.

\medskip
L'énoncé suivant précise en particulier l'appartenance du résultant
à l'idéal d'élimination puisque tous les annulateurs qui interviennent
sont contenus dans $\ElimIdeal$ (fait facile mentionné en \ref{AnnEqualities}).
On voit également comment le résultant s'intercale 
dans la double inclusion classique entre $0$-Fitting et annulateur
du module de présentation finie $\bfB'_\delta$ et/ou $\bfB_d$.

\begin {coro}
Soit $\uP$ un système quelconque.

\begin {enumerate} [i)]
\item
En degré $\delta$, on dispose des inclusions:
$$
\calF_0(\bfB'_\delta) \ \subset\ \langle\omegares(\nabla) \rangle \ \subset\
\Ann(\bfB'_\delta) \ \subset\ \sqrt{\calF_0(\bfB'_\delta)}
$$
\item
En degré $d \geqslant \delta + 1$, de manière analogue:
$$
\calF_0(\bfB_d) \ \subset\ \langle\calR_d\rangle \ \subset\
\Ann(\bfB_d) \ \subset\ \sqrt{\calF_0(\bfB_d)}
$$
\end {enumerate}

\smallskip
\noindent
En conséquence, tous les idéaux ci-dessus ont même racine $\sqrt{\Res(\uP)}$.
\end {coro}

\begin {proof} \leavevmode

Dans les deux cas, on va utiliser le fait que l'annulateur
d'un module de présentation finie est contenu dans la racine de son $0$-Fitting.

\smallskip
i) L'appartenance $\omegares(\nabla) \in \Ann(\bfB'_\delta)$ 
résulte du théorème \ref{omegaresNablaInElimIdeal} et le reste des
divers rappels.

\medskip

ii) L'appartenance $\calR_d \in \Ann(\bfB_d)$ résulte du fait que
$\calR_d = \omegares(\nabla)$, de
$\omegares(\nabla) \in \Ann(\bfB'_\delta)$ et enfin de l'inclusion
banale $\Ann(\bfB'_\delta) \subset \Ann(\bfB_d)$ vue en~\ref{AnnEqualities}.
Pour le reste, cf les rappels.
\end {proof}

\medskip

Dans le même ordre d'idées, sans aucune hypothèse spécifique sur $\uP$, voici
comment se positionne l'idéal d'élimination par rapport à l'idéal engendré
par le résultant.

\begin{coro} \label{ResultantVersusElimIdeal}
Pour un système quelconque $\uP$, on dispose des inclusions:
$$
\langle \Res(\uP) \rangle
\ \subset \ 
\ElimIdeal 
\ \subset \ 
\sqrt{\Res(\uP)}
$$
L'idéal d'élimination a donc même radical que l'idéal engendré par le résultant:
$\sqrt{\ElimIdeal} = \sqrt{\Res(\uP)}$.
\end{coro}

\begin{proof} \leavevmode

L'appartenance $\Res(\uP) \in \ElimIdeal$ résulte du
théorème~\ref{omegaresNablaInElimIdeal}.  En ce qui concerne la seconde
inclusion, soit $a \in \ElimIdeal$. Nous allons fournir deux preuves
du fait qu'une certaine puissance de $a$ est multiple du résultant.

\medskip
La première preuve utilise le fait que l'annulateur d'un module de présentation finie
est contenu dans la racine de son 0-Fitting. 
Par définition de $\ElimIdeal$, il y a un $d$, que l'on peut supposer
$\geqslant\delta+1$, tel que $X^\alpha a \in \langle\uP\rangle$ pour
tout $|\alpha| = d$, ce qui s'écrit $a \in \Ann(\bfB_d)$.  Puisque
$\Ann(\bfB_d) \subset \sqrt{\calF_0(\bfB_d)}$, il y a une puissance de
$a$ dans $\calF_0(\bfB_d)$, ce qui termine la preuve puisque tout
élément de $\calF_0(\bfB_d)$ est multiple de~$\Res(\uP)$ vu que
$d \geqslant\delta+1$.

\medskip

Pour la seconde preuve (tirée de~\cite[Proposition 2.4]{BuseJouanolouDiscriminantI}),
nous allons utiliser la propriété de divisibilité du résultant énoncée
en~\ref{PremierLemmeDivisibiliteJPJ}. Par définition de $\ElimIdeal$,
il y a des exposants $e_i \geqslant 0$ tels que
$a X_i^{e_i} \in \langle \uP \rangle$.
Par la propriété de divisibilité appliquée à $Q_i = a X_i^{e_i}$, on obtient
$$
\Res(a X_1^{e_1}, \dots, a X_n^{e_n}) 
\ \in \
\langle \Res(\uP) \rangle
$$
Or, en notant $\widehat {e_i} = \prod_{j\ne i} e_j$, on a $\Res(a_1
X_1^{e_1}, \dots, a_n X_n^{e_n}) = a_1^{\widehat e_1} \cdots
a_n^{\widehat e_n}$.  Ce qui fournit ici l'appartenance
$a^{\widehat e_1+\cdots+\widehat e_n} \in \langle \Res(\uP) \rangle$.
\end{proof}

\medskip

Nous allons de nouveau comparer les idéaux attachés au degré $\delta$,
intervenus dans le chapitre~\ref{ChapMacRaeForP}: d'une part
$\Im\,\omegares$ et d'autre part l'idéal déterminantiel
$\calD_{s_\delta}(\Syl_\delta)$ d'ordre $s_\delta
= \dim\bfA[\uX]_\delta-1$. Quel que soit le système $\uP$, cet idéal
déterminantiel se factorise sous la forme
(cf. page~\pageref{ResumeSpecialisationCasGenerique}):
$$
\calD_{s_\delta}(\Syl_\delta) = (\Im\omegares)\fb_\delta
$$

\begin {theo}
\label{omegaresVersusSyldelta}
Pour tout système $\uP$, on dispose de l'inclusion et de l'égalité des radicaux:
$$
\calD_{s_\delta}(\Syl_\delta)  \subset \Im\,\omegares, \qquad 
\sqrt{\calD_{s_\delta}(\Syl_\delta)}  = \sqrt{\Im\,\omegares}
$$
En particulier, les deux idéaux en question ont même profondeur.

\medskip

Précision: pour $a\in \Im\,\omegares$, on a $a^{s_\delta} \in \calD_{s_\delta}(\Syl_\delta)$.
\end {theo}

\begin {proof}

En ce qui concerne l'inclusion (qui ne nécessite aucune hypothèse sur
$\uP$), voir en page~\pageref{ResumeSpecialisationCasGenerique} (cf
également en début de preuve du théorème \ref{omegaresGr2}).

\medskip

Comme l'égalité des radicaux se spécialise, on peut supposer $\uP$
générique.  On procède comme dans la preuve du
théorème \ref{omegaresGr2} en augmentant le complexe
$\rmK_{\sbullet,\delta}(\uP)$ par $\omegares
: \rmK_{0,\delta} \to \bfA$.  Le complexe augmenté reste exact (car on
s'est placé en terrain générique) et, en utilisant les
notations \ref{IdeauxFactorisation} pour ses idéaux de factorisation,
on a, en tenant compte du décalage dû à l'augmentation:
$$
\fD_2 = \fB_2\fB_3  \qquad \text{avec} \qquad
\fD_2 = \calD_{s_\delta}(\Syl_\delta), \qquad \fB_2 = \Im\,\omegares
$$
Le théorème \ref{RacineIdealFactorisation} (\og unexpected result\fg{}
selon Northcott) permet de conclure: $\sqrt {\fD_2} = \sqrt {\fB_2}$.

\medskip

La justification du résultat beaucoup plus précis,
$a^{s_\delta} \in \calD_{s_\delta}(\Syl_\delta)$ pour $a\in
\Im\,\omegares$, nécessite la mise en place des outils 
qui suivent et est reportée après la preuve du corollaire~\ref{TracicIdentity}.
\end {proof}

\medskip

La vie (mathématique) réserve parfois des surprises. La première
surprise est que le résultat plus précis est basé sur le théorème
suivant et son corollaire, étrangers au contexte de l'élimination.  La
seconde surprise réside dans le fait que ce théorème ne sort pas du
chapeau d'un magicien mais de l'analyse d'un cas trivial de
factorisation (thème du
chapitre~\ref{ChapStructureMultiplicative}).  La dernière surprise,
c'est l'utilisation du produit intérieur (droit) qui permet de ne pas
sombrer dans des calculs épouvantables.

\begin {theo} [Factorisation explicite et comment s'invite le produit intérieur droit]
\label{ExplicitFactorisation}

Soient $F$ un module libre de rang $p$, $\mu \in F^\star$ et $y \in F$.
On pose \boxed{s = p-1} et on note $y\cdot\mu$ l'endomorphisme de~$F$ défini par $x \mapsto \mu(x)y$,
si bien que $\varphi := \mu(y)\Id_F - y\cdot\mu$ est l'endomorphisme de $F$:
$$
x \mapsto \mu(y)x - \mu(x)y
$$

\begin {enumerate}[\rm i)]
\item
En fixant une orientation $\bff$ sur $F$,  on a la factorisation  
$$
\BW^s(\varphi) = \mu(y)^{s-1}\ \Theta_\mu \cdot \nu_y \qquad \text{où} \quad
\left\{\begin {array} {l}
\Theta_\mu = \bff\intd\mu, \text { vecteur de  } \BW^s(F) \\
\nu_y = \bff^\star \intd y, \text { forme linéaire sur } \BW^s(F) \\
\end {array}\right.
$$
De plus $\nu_y(\Theta_\mu) = \mu(y)$.

\item
Une version détypée consiste à prendre $F = \bfA^p$, sa base canonique
$f = (f_1, \cdots, f_p)$, 
un vecteur $y \in \bfA^p$ jouant le même rôle et $\mu$ sous la
forme $\transpose{z}$ pour un $z \in \bfA^p$. En définissant $Y, Z \in \BW^s(\bfA^p)$:
$$
Y = \sum_i (-1)^{i-1} y_i\, f_{\{1..p\} \setminus i},  \qquad
Z = \sum_i (-1)^{i-1} z_i\, f_{\{1..p\} \setminus i} \qquad \text{avec} \qquad
f_I = \BW\limits_{i\in I} f_i
$$
la factorisation s'écrit alors:
$$
\BW^s\big(\scp{y}{z}\Id_p -y\cdot\transpose {z}\big) = \scp{y}{z}^{s-1} \
Z\cdot\transpose{Y}
$$
De plus $\scp{Y}{Z} = \scp{y}{z}$.
\end {enumerate}
\end {theo}

Pour $p = 3$, ce résultat dit la chose concrète suivante. Considérons la version détypée
et posons $a = \scp{y}{z}$. L'endomorphisme $\varphi := a\Id_3 - y\cdot \transpose{z}$ de $\bfA^3$
est donné sur la base canonique $(f_1, f_2, f_3)$ par
$$
\varphi(f_i) = af_i - z_i(y_1f_1 + y_2f_2 + y_3f_3)
$$
Posons $f_{ij} = f_i\wedge f_j$. Il faut avoir le courage de calculer les
$\BW^2(\varphi)(f_{ij}) = \varphi(f_i) \wedge \varphi(f_j)$:
$$
\left\{
\begin {array}{ccl}
\varphi(f_2)\wedge \varphi(f_3) &=& ay_1 (z_1 f_{23} - z_2f_{13} + z_3f_{12}) \\
\varphi(f_1)\wedge \varphi(f_3) &=& ay_2 (-z_1 f_{23} + z_2f_{13} - z_3f_{12}) \\
\varphi(f_1)\wedge \varphi(f_2) &=& ay_3 (z_1 f_{23} - z_2f_{13} + z_3f_{12}) \\
\end {array}
\right.
$$
Ce  qui conduit à la matrice de $\BW^2(\varphi)$ dans la base $(f_{23},\ f_{13},\ f_{12})$
avec la factorisation escomptée:
$$
\BW^2(\varphi) = a
\begin {bmatrix}
z_1 \times y_1  & z_1 \times -y_2  & z_1 \times y_3 \\
-z_2 \times y_1 & -z_2\times -y_2  & -z_2 \times y_3 \\
z_3 \times y_1  & z_3\times -y_2   & z_3 \times y_3 \\
\end {bmatrix}
=
a\, \begin {bmatrix} z_1\cr -z_2\cr z_3 \end {bmatrix} \
\begin {bmatrix} y_1 & -y_2 & y_3 \end {bmatrix}
$$
Nous proposons deux preuves de ce théorème. La première est basée sur
le lemme suivant.

\begin {lem}

Soit $F$ un $\bfA$-module quelconque, $\mu$ une forme linéaire sur $F$ et $y\in F$.

\begin {enumerate}[\rm i)]
\item
Pour $z \in F$ et $\bfx \in \BW^k(F)$, on a l'égalité:
$$
\big((y\wedge\bfx)\intd\mu\big) \wedge \big(\mu(y)z - \mu(z)y\big) =
\mu(y)\, (y\wedge\bfx\wedge z)\intd\mu
$$

\item
Pour $z_1, \ldots, z_m \in F$, on a l'égalité dans $\BW^m(F)$:
$$
\big(\mu(y)z_1 - \mu(z_1)y\big) \wedge \cdots \wedge \big(\mu(y)z_m - \mu(z_m)y\big) =
\mu(y)^{m-1}\, (y \wedge z_1 \wedge z_2 \wedge \cdots \wedge z_m) \intd \mu
$$
\end {enumerate}
\end {lem}

\begin {proof} \leavevmode

Pour un élément homogène $\bfy$ de $\BW^\sbullet(F)$, on note
$J(\bfy) = (-1)^{\deg \bfy} \bfy$.

\medskip
i) Le membre droit de l'égalité à montrer est de la forme $\mu(y)A$ avec
$A = (y\wedge\bfx\wedge z)\intd\mu$.
Utilisons le fait que $\sbullet\intd\mu$ est une anti-dérivation à gauche pour transformer $A$:
$$
\begin {array} {ccl}
A &=& \big((y\wedge\bfx) \intd\mu\big) \wedge z + J(y\wedge\bfx)\wedge (z\intd\mu) \\
  &=& B\wedge z + (-1)^{k+1} \mu(z)\, y\wedge\bfx \\
\end {array}
\qquad\quad
\hbox {avec $B = (y\wedge\bfx) \intd\mu$}
\leqno(\star)
$$
On fait ainsi apparaître $B$ qui intervient dans le membre gauche de l'égalité de l'énoncé.
En multipliant~$(\star)$ par $\mu(y)$, on obtient:
$$
\mu(y)\,A = B\wedge \mu(y) z + (-1)^{k+1} \mu(z)\,\mu(y)\,y\wedge\bfx 
$$
Montrons maintenant que $\mu(y)\, y\wedge\bfx = y\wedge B$. En
utilisant de nouveau que $\sbullet\intd\mu$ est une anti-dérivation à
gauche pour transformer $B$:
$$
B = (y \wedge \mu) \wedge\bfx + J(y)\wedge(\bfx\intd\mu)
\qquad\text{d'où}\qquad
y\wedge B = \mu(y)\, y\wedge\bfx
$$
En intégrant cette information dans l'égalité $\mu(y)A = \cdots$ 
et en utilisant le fait que $B$ est de degré $k$:
$$
\begin {array} {ccl}
\mu(y)\, A &=& B\wedge \mu(y) z + (-1)^{k+1} \mu(z)\, y \wedge B =
B\wedge \mu(y) z + (-1)^{k+1} (-1)^k \mu(z)\, B \wedge y \\
&=& B\wedge \big(\mu(y)z - \mu(z)y\big) \\
\end {array}
$$
C'est exactement ce que l'on veut.

\medskip
ii)
Par récurrence sur $m$, en posant $\bfx = z_1 \wedge z_2 \wedge \cdots \wedge z_{m-1}$:
$$
\big(\mu(y)z_1 - \mu(z_1)y\big) \wedge \cdots \wedge \big(\mu(y)z_{m-1} - \mu(z_{m-1})y\big) =
\mu(y)^{m-2}\, (y \wedge \bfx) \intd \mu
$$
On termine avec le point i).
\end {proof}

\begin {proof} [Première preuve du théorème \ref{ExplicitFactorisation}]

Examinons l'égalité à prouver. A gauche,
en $z_2 \wedge \cdots \wedge z_p$, l'endomorphisme $\BW^s(\varphi)$ vaut:
$$
A =
\big(\mu(y)z_2 - \mu(z_2)y\big) \wedge \cdots \wedge \big(\mu(y)z_p - \mu(z_p)y\big)
$$
tandis que l'évaluation du membre droit est:
$$
B = 
\mu(y)^{s-1}\ \nu_y(z_2 \wedge \cdots \wedge z_p)\cdot\Theta_\mu =
\mu(y)^{s-1}\ [y \wedge z_2 \wedge \cdots \wedge z_p]_\bff\,(\bff\intd\mu)
$$
Il s'agit de montrer l'égalité $A=B$ dans $\BW^s(F)$. D'après le point ii)
du lemme précédent, on a l'égalité $A = \mu(y)^{s-1}\, A'$ avec
$$
A' = (y \wedge z_2 \wedge \cdots \wedge z_p) \intd \mu
$$
A droite, on a $B = \mu(y)^{s-1}\ B'$ avec
$$
B' = [y \wedge z_2 \wedge \cdots \wedge z_p]_\bff\,(\bff\intd\mu) \buildrel {\rm (*)} \over =
(y \wedge z_2 \wedge \cdots \wedge z_p) \intd\mu
$$
l'égalité $(*)$ résultant de l'égalité $y \wedge z_2 \wedge \cdots \wedge z_p =
[y \wedge z_2 \wedge \cdots \wedge z_p]_\bff\,\bff$. Bilan: $A'=B'$ donc $A=B$,
ce qui termine la preuve.
\end {proof}

\begin {proof} [Seconde preuve de la factorisation du théorème \ref{ExplicitFactorisation}]\leavevmode

Nous reprenons les notations $A,B$ en début de la première preuve, pour lesquelles
nous devons montrer l'égalité $A=B$ dans $\BW^s(F)$.
A gauche, il s'agit de développer le produit extérieur de $s$ vecteurs:
$$
A \ =\
\overbrace {(\mu(y)z_2 - \mu(z_2)y)}^{(2)} \ \wedge\
\overbrace {(\mu(y)z_3 - \mu(z_3)y)}^{(3)} \ \wedge\
\overbrace {(\mu(y)z_4 - \mu(z_4)y)}^{(4)} \ \wedge\
\cdots
$$
Puisque $y \wedge y = 0$, mis à part le terme spécial 
$\mu(y)^s\ z_2 \wedge z_3 \wedge z_4 \wedge \cdots$, on 
sélectionne un seul $\mu(z_i)y$ par parenthèse, ce qui réalise
$A$ comme la somme de:
$$
\begin {array} {ccccccccc}
\mu(y)^{s-1}  &\mu(y) &z_2&\wedge &z_3&\wedge&z_4& \cdots &
\\
-\mu(y)^{s-1} &\mu(z_2) &y &\wedge&z_3&\wedge&z_4 &\cdots
     &\hbox {(on a pris $-\mu(z_2)y$ dans $(2)$)}  \\
-\mu(y)^{s-1} &\mu(z_3) &z_2&\wedge&y&\wedge&z_4 &\cdots
     &\hbox {(on a pris $-\mu(z_3)y$ dans $(3)$)}  \\
-\mu(y)^{s-1} &\mu(z_4) &z_2&\wedge&z_3&\wedge&y &\cdots
     &\hbox {(on a pris $-\mu(z_4)y$ dans $(4)$)}  \\
\end {array}
$$
ou encore $A = \mu(y)^{s-1}\, A'$ avec
$$
\begin {array} {ccl}
A' &=& \mu(y)\,z_2\wedge z_3\wedge z_4\wedge\cdots 
- \mu(z_2)\,y\wedge z_3\wedge z_4\wedge\cdots 
+ \mu(z_3)\,y\wedge z_2\wedge z_4\wedge\cdots 
- \mu(z_4)\,y\wedge z_2\wedge z_3\wedge\cdots
\\
  &\overset{\rm def.}{=}& (y \wedge z_2 \wedge z_3 \wedge z_4 \wedge \cdots \wedge z_p) \intd \mu
\\
\end {array}
$$
Cette expression de $A'$, donc de $A$, est celle obtenue dans la première preuve.
On termine alors comme dans cette première preuve.
\end {proof}

\begin {coro} [Identité tracique et idéal déterminantiel] \leavevmode
\label {TracicIdentity}

Le contexte est celui du théorème \ref{ExplicitFactorisation} dans lequel intervient
l'endomorphisme $\varphi$ de $F$ égal à $\mu(y)\Id_F -y\cdot\mu$:
$$
\varphi : x\mapsto\mu(y)x - \mu(x)y
$$
En rappelant que $s = \dim F - 1$, on a l'identité tracique:
$$
\Tr \big(\BW^s \varphi\big) = \mu(y)^s
$$
En particulier, le scalaire  $\mu(y)^s$ appartient l'idéal
déterminantiel d'ordre $s$ de l'endomorphisme~$\varphi$:
$$
\mu(y)^s \in \calD_s(\varphi)
$$
Pour être totalement explicite, on note $f_I = \BW\limits_{i \in I} f_i$ où
$(f_i)$ est une base de $F$ et $(f^\star_i)$ la base duale. Alors:
$$
\mu(y)^s = \sum_{\#I=s} \scp{v_I}{f^\star_I}
\qquad \text{où} \qquad
v_I = \Big(\BW^s \varphi\Big)(f_I) = \BW\limits_{i\in I} \big(\mu(y)f_i - \mu(f_i)y\big)
$$
\end {coro}

\begin {proof}
Il suffit de prendre la trace dans la factorisation intervenant dans le théorème et d'utiliser:
$$
\Tr(\Theta_\mu \cdot \nu_y) = \nu_y(\Theta_\mu) = \mu(y)
$$
\end {proof}

\begin {proof} [Fin de la preuve du théorème \ref{omegaresVersusSyldelta}]

Nous devons montrer que $a^{s_\delta} \in \calD_{s_\delta}(\Syl_\delta)$ avec
$a \in \Im\omegares$, disons $a = \omegares(H)$ où $H \in \bfA[\uX]_\delta$.
On prend $F=\bfA[\uX]_\delta$, $\mu = \omegares$ et $y=H \in \bfA[\uX]_\delta$.
L'entier $s=p-1$ vaut~$s_\delta$ et l'endomorphisme $\varphi$ de $\bfA[\uX]_\delta$ est
$$
\varphi \overset{\rm def.}{=} \omegares(H)\Id_{\bfA[\uX]_\delta} - \omegares\,H
\qquad\quad
G \mapsto \omegares(H)G - \omegares(G)H
$$
La propriété Cramer de $\omegares$ se traduit par
l'inclusion $\Im \varphi \subset \langle\uP\rangle_\delta \overset{\rm def}{=}
\Im\Syl_\delta$. On en déduit: 
$$
\Im\BW^{s_\delta}(\varphi) \subset \Im \BW^{s_\delta}(\Syl_\delta) \qquad \text{donc} \qquad
\calD_{s_\delta}(\varphi) \subset \calD_{s_\delta}(\Syl_\delta)
$$
D'après le corollaire ci-dessus, on a $\omegares(H)^{s_\delta} \in \calD_{s_\delta}(\varphi)$,
a fortiori l'appartenance:
$$
\omegares(H)^{s_\delta} \in \calD_{s_\delta}(\Syl_\delta)
$$
qui avait été annoncée.
\end {proof}

\subsection {A quelle condition le résultant engendre-t-il l'idéal d'élimination?}

Nous allons de nouveau utiliser les propriétés de $\omegares$, ou de
$\uomegares$ si l'on veut, pour répondre à la question posée dans le
titre de cette section. Mais quelles propriétés exactement? Nous
allons voir qu'il s'agit d'une part de la propriété Cramer en le
bezoutien $\nabla$ de $\uP$ et d'autre part de l'inégalité de profondeur
$\Gr(\omegares) \geqslant 2$.  Nous rappelons à cette occasion que
$\uomegares$ est de Cramer (et pas seulement en le bezoutien) 
et ceci sans aucune hypothèse sur $\uP$ (cf. l'énoncé~\ref{omegaresCramer}).
Et que l'inégalité $\Gr(\omegares) \geqslant 2$ est vérifiée pour la suite $\uP$
générique (cf. l'énoncé~\ref{omegaresGr2}).

\medskip
Afin de nous convaincre de la suffisance de ces hypothèses, nous
fournissons un premier résultat abstrait dans le contexte 
$(M, \vartheta, m_0)$ devenu habituel et rappelé en page \pageref{RappelCramerAnn}.
Mais que le lecteur ne s'y trompe pas car nous y
avons mis les hypothèses ad-hoc: $\vartheta$ est de Cramer en $m_0$ et
vérifie $\Gr(\vartheta) \geqslant 2$!
S'il le souhaite, il peut lire ce résultat dans le contexte de
l'élimination, que nous énonçons de toute manière à la suite.

\begin{theo}
\leavevmode

Soit $M$ un module de MacRae de rang $1$, $\vartheta
\in M^\star$ un générateur de son invariant de MacRae $\MacRaeVect(M)$ et $M'$
le module de MacRae de rang $0$ défini par $M' := M/\bfA m_0$ où $m_0 \in M$ est sans torsion.

\smallskip
\noindent
On suppose que la forme linéaire $\vartheta$ est de Cramer en $m_0$ et
qu'elle vérifie $\Gr(\vartheta) \geqslant 2$.

\begin{enumerate}[\rm i)]
\item 
La forme linéaire $\vartheta$ est une base de $M^\star$ et $M^\star = \Cramer_{m_0}(M) = \bfA\vartheta$.

\item 
L'évaluation en $m_0$, $\alpha \mapsto \alpha(m_0)$, 
réalise un isomorphisme $M^\star \buildrel{\simeq}\over \longrightarrow \Ann(M')$.

\item
L'idéal $\Ann(M')$ est engendré par son invariant de MacRae $\MacRae(M') = \vartheta(m_0)$.

\item
L'invariant de MacRae $\MacRae(M')$ appartient à $\sqrt{\calF_0(M')}$.
En conséquence, $\sqrt{\MacRae(M')} = \sqrt{\calF_0(M')}$.
\end{enumerate}
\end{theo}

\begin{proof}
i) 
Le point ii) de~\ref{FormesLineairesRang01} s'applique à la forme linéaire $\vartheta$
qui, par hypothèse, vérifie $\Gr(\vartheta) \geqslant 2$. Ce résultat nous dit
que $M^\star$ est libre de rang $1$ engendré par $\vartheta$.
Comme on a supposé $\vartheta$ de Cramer en~$m_0$, on obtient donc les
inclusions
$$
\bfA \vartheta \ \subset \ \Cramer_{m_0}(M) \subset \ 
M^\star = \bfA \vartheta
$$
En conséquence, il y a égalité partout.

\smallskip

ii)
D'après le rappel en début de chapitre
(page~\pageref{RappelCramerAnn}), l'évaluation en $m_0$ établit un
isomorphisme entre $\Cramer_{m_0}(M)$ et $\Ann(M')$. Et d'après i),
$\Cramer_{m_0}(M) = M^\star$.

\smallskip

iii) La forme linéaire $\vartheta$ est une base de $\Cramer_{m_0}(M)$. 
D'après ce même rappel, l'invariant $\MacRae(M')$, égal à $\vartheta(m_0)$, 
est un générateur de l'idéal $\Ann(M')$. 

\smallskip

iv) Résulte du fait que $\MacRae(M')$ est un générateur de  $\Ann(M')$
et du fait que les idéaux $\Ann(M')$, $\calF_0(M')$ ont même racine.
\end{proof}

\medskip

Tout est prêt pour énoncer le résultat principal. Insistons encore que
dans notre approche, nous avons choisi de définir le résultant de $\uP$
comme étant l'évaluation $\omegares(\nabla)$ en un
bezoutien $\nabla$ de $\uP$. Il est donc impératif d'avoir en tête
que nous parlons d'un même objet sous trois noms différents:
$$
\MacRae(\bfB'_\delta) \ = \  \omegares(\nabla) \ = \ \Res(\uP)
$$

\begin{theo} 
\label{4pointsPsuperreguliere}
\leavevmode

Soit $\uP$ régulière vérifiant $\Gr(\omegares) \geqslant 2$
(c'est le cas par exemple de la suite générique).

\begin{enumerate}[\rm i)]
\item 
Le $\bfA$-module $\bfB_\delta^\star$ est libre de rang 1, de base $\uomegares$,
invariant de MacRae de $\bfB_\delta$.
Remonté à l'anneau de polynômes $\bfA[\uX]$, cela s'énonce : 
la forme linéaire $\omegares$ est une base de $\Ker \transpose {\Syl_\delta}$.

\item 
L'évaluation en un bezoutien $\nabla$ de $\uP$ est un isomorphisme:
$$
\begin{array}[t]{rcl}
\bfB_\delta^\star & \buildrel {\simeq} \over \longrightarrow & \Ann(\bfB'_\delta) \\ [0.2cm]
\overline{\mu} & \longmapsto & {\overline \mu}(\overline {\nabla})
\end{array}
\qquad \text{ ou encore } \qquad 
\begin{array}[t]{rcl}
\Ker \transpose \Syl_\delta & \buildrel {\simeq} \over \longrightarrow & \ElimIdeal \\ [0.2cm]
\mu & \longmapsto & \mu(\nabla)
\end{array}
$$

\item
On a l'égalité $\Ann(\bfB'_\delta) = \ElimIdeal$ et l'invariant de MacRae
$\MacRae(\bfB'_\delta) = \omegares(\nabla)$ est un générateur régulier
de cet idéal.

\item
Pour $d \geqslant \delta+1$, on a $\Ann(\bfB_d) = \ElimIdeal$, 
et l'invariant de MacRae $\MacRae(\bfB_d) = \calR_d$ est un générateur régulier de cet idéal.

\end{enumerate}
\end{theo}

\begin{proof}
L'hypothèse $\uP$ régulière assure que $\bfB_\delta$ est de MacRae de rang $1$ 
et que $\Ann(\bfB'_\delta) = \ElimIdeal$ (cf. le deuxième point de~\ref{AnnEqualities}).
La forme linéaire $\uomegares$  est toujours de Cramer (cf.~\ref{omegaresCramer}),
a fortiori de Cramer en~$\overline \nabla$.
L'hypothèse $\Gr(\omegares)\geqslant 2$ imposée dans l'énoncé permet d'appliquer
le théorème précédent, ce qui fournit les 3 premiers points.

\medskip
Note: en ce qui concerne le point i), cf une variante dans la \emph{preuve}
du point ii) de~\ref{omegaresGr2} qui utilise uniquement $\Gr(\uomegares) \ge 2$.

\medskip

Quant au dernier point iv), il se justifie par le fait que, pour une
suite régulière, on dispose de l'égalité des annulateurs (cf. le
deuxième point de~\ref{AnnEqualities}), annulateurs tous égaux à
$\ElimIdeal$
$$
\Ann(\bfB'_{\delta}) = \Ann(\bfB_{\delta+1})
\qquad \text{ et } \qquad 
\forall\, d \geqslant \delta+1, \quad 
\Ann(\bfB_d) = \Ann(\bfB_{d+1})
$$
\end{proof}

\subsection{En générique, le résultant engendre l'idéal d'élimination}

On reformule le résultat principal avec le vocabulaire \og résultant
\fg{}.  Le premier point n'est pas spécifique au cas
générique. Appliqué au cas générique, le fait que le résultant est un
générateur de l'idéal d'élimination contribue à apporter une précision sur
l'inégalité $\Gr(\omegaRes{\uP^\gen}) \ge 2$.

\begin{theo} \label{ResGenElimIdeal}\leavevmode

\begin {enumerate}[\rm i)]
\item
Soit $\uP$ une suite régulière vérifiant $\Gr(\omegares) \geqslant 2$, par exemple la suite générique.
Alors, l'idéal d'élimination est monogène, engendré par le résultant :
$$
\ElimIdeal \ = \  \langle \Res(\uP) \rangle
$$
\item
Quand $\uP$ est générique, pour tout monôme $X^\alpha$ de degré $\delta$, la suite
$\big(\Res(P), \omegares(X^\alpha)\big)$ est une suite régulière dans l'image de $\omegares$,
précisant ainsi l'inégalité $\Gr(\omegares) \ge 2$.
\end {enumerate}
\end{theo}

\begin {proof} \leavevmode
ii)
Nous reprenons le commentaire figurant avant le théorème \ref{PsatReg}.
On a vu en~\ref{GenPsatRegScalars} que~$\omega(X^\alpha)$ est
régulier sur $\bfA[\uX]/\uPsat$, a fortiori sur le sous-anneau
$\bfA/(\ElimIdeal)$.  Il en est de même de~$\omegares(X^\alpha)$ qui
est un diviseur de $\omega(X^\alpha)$. Mais en générique, l'idéal
$\ElimIdeal$ est engendré par $\Res(\uP)$ et donc la suite
$\big(\Res(P), \omegares(X^\alpha)\big)$ est régulière, bien sûr dans
l'idéal $\Im\omegares$ puisque $\Res(\uP)=\omegares(\nabla)$.
\end {proof}

\medskip

L'idéal d'élimination n'est pas en général engendré par le résultant, même pour une suite régulière.
Par exemple, pour le jeu étalon généralisé $\pXD = (p_1 X_1^{d_1},\dots,p_n X_n^{d_n})$, on a vu
dans la remarque~\ref{rmqMacRaeVsElimIdeal}:
$$
\langle \pXD \rangle^{\sat} \cap \bfA \ = \ 
\langle p_1 \cdots p_n \rangle
\qquad \text{ tandis que } \qquad 
\Res(\pXD) 
\ = \ 
p_1^{\widehat d_1}\cdots p_n^{\widehat d_n}
$$

\begin{prop}[Caractère \og irréductible\fg{} du résultant] \label{IrreductibiliteResultant}
En générique au-dessus d'un anneau de base $\bfk$ intègre, 
le résultant est irréductible au sens où il engendre un idéal premier.
\end{prop}

\begin{proof}
Résulte de $\ElimIdeal = \langle \Res(\uP) \rangle$ et
du fait que $\uPsat$ est un idéal premier, cf.~\ref{GeneriqueLocaliseSature}-iii).
\end{proof}

\subsection{Propriétés du résultant, via \og générateur de l'idéal d'élimination \fg{}}

Les résultats ci-dessous ont déjà fait l'objet du chapitre
précédent, cf~\ref{PremierLemmeDivisibiliteJPJ} et~\ref{MultiplicativiteResultant}.
Nous les abordons ici de manière différente en utilisant le fait qu'en
terrain générique le résultant est un générateur de l'idéal
d'élimination, cf \ref{ResGenElimIdeal} et qu'il est
irréductible au dessus de $\bbZ$, cf~\ref{IrreductibiliteResultant}.

\begin{prop}[Divisibilité]
Soient $\uP = (P_1,\dots,P_n)$ et $\uQ = (Q_1,\dots,Q_n)$ deux suites 
de polynômes homogènes de $\bfA[\uX]$ avec leur propre format de degrés.
Si $\langle \uQ \rangle \subset \langle \uP \rangle$ 
alors $\Res(\uP) \mid \Res(\uQ)$.
\end{prop}

\index{théorème!de divisibilité du résultant}%
%
%

\begin{proof}
On écrit $Q_j = \sum_i U_{ij} P_i$
avec $U_{ij} \in \bfA[\uX]$ homogène de degré $\deg(Q_j) - \deg(P_i)$.
On note $\uP^\gen$ le système générique au dessus de $\bfA$ de même format que $\uP$
, de sorte que $P_i^{\gen}$ est à coefficients 
dans $\bfA' = \bfA[\text{indets pour les\ } P_i^\gen]$.
On introduit $Q'_j = \sum_i U_{ij} P_i^\gen \in \bfA'[\uX]$.
On a 
$$
\Res(\uQ') \ \in \ 
\langle \uQ' \rangle^\sat \cap \bfA' 
\ \subset \ 
\langle \uP^\gen \rangle^\sat \cap \bfA' 
\ = \ 
\langle \Res(\uP^\gen) \rangle_{\bfA'}
$$
En effet, l'appartenance à gauche est banale, l'inclusion du milieu 
est due à $\langle \uQ' \rangle \subset \langle \uP^\gen \rangle$ ;
l'égalité à droite résulte du fait que $\uP^\gen$ est générique.
Ainsi $\Res(\uP^\gen) \mid_{\bfA'} \Res(\uQ')$ et 
en spécialisant on obtient
$\Res(\uP) \mid_{\bfA} \Res(\uQ)$.
\end{proof}

\begin{prop}[Multiplicativité]
Soit $P_1, P'_1, P_2, \ldots, P_n$ des polynômes homogènes de degrés $d_1, d'_1, d_2, \ldots, d_n$. 
On a alors l'égalité :
$$
\Res(P_1P'_1,\, P_2, \ldots, P_n) 
\ =\  
\Res(P_1, P_2, \ldots, P_n) \ \Res(P'_1, P_2, \ldots, P_n)
$$
\end{prop}

\index{théorème!de multiplicativité!du résultant}%

\begin{proof}
Notons :
$$
\uP = (P_1, P_2, \ldots, P_n), \qquad \uP' = (P'_1, P_2, \ldots, P_n), \qquad
\uQ = (P_1P'_1, P_2, \ldots, P_n)
$$
Il suffit de se placer en terrain générique au-dessus de $\bbZ$ en
considérant l'anneau de polynômes $\bfA = \bbZ[\uT,\uU]$ où $\uT$ est une
famille d'indéterminées allouée aux coefficients des $P_i$ et $\uU$
à ceux de $P'_1$. Ainsi les suites $\uP$ et~$\uP'$ sont génériques
et nous allons pouvoir utiliser l'arithmétique des anneaux de
polynômes sur $\mathbb Z$. On note $p_i$ le coefficient en $X_i^{d_i}$
de $P_i$ et $p'_1$ celui en $X_1^{d'_1}$ de $P'_1$.

\smallskip
\noindent
Comme $\langle\uQ\rangle \subset \langle\uP\rangle$, le théorème  de divisibilité du résultant
fournit $\Res(\uP) \mid \Res(\uQ)$ et de la même manière, $\Res(\uP') \mid \Res(\uQ)$.

\smallskip
\noindent
D'après~\ref{IrreductibiliteResultant}, $\Res(\uP)$ et $\Res(\uP')$ sont
premiers. De plus, ils ne sont pas associés : en effet, en notant $e_1
= d_2 \cdots d_n$, $\Res(\uP)$ est dans $\bbZ[\uT]$ et contient un
monôme du type $(p_2^{e_2} \cdots p_n^{e_n}) p_1^{e_1}$ tandis que
$\Res(\uP')$ contient un monôme du type $(p_2^{e_2'} \cdots
p_n^{e_n'}) {p'_1}^{e_1}$.  On en déduit~:
$$
\Res(\uP)\,\Res(\uP') \mid  \Res(\uQ)
$$
On vérifie que $\Res(\uP)\Res(\uP')$ et $\Res(\uQ)$ ont même poids en les coefficients de 
$P_2, \dots, P_n$ et du produit $P_1P'_1$.
Ainsi il existe $u$ \textit{dans $\bbZ$} tel que 
$\Res(\uQ) = u \Res(\uP)\Res(\uP')$.
Avec les spécialisations $P_i := X_i^{d_i}$ et $P'_1 := X_1^{d'_1}$, 
on obtient $u=1$, ce qui termine la preuve.
\end{proof}

\begin{prop}[Résultant de la suite permutée]
Soit $\sigma \in \fS_n$. En notant $\epsilon(\sigma)$ la signature de $\sigma$, on a 
l'égalité 
$$
\Res(P_{\sigma(1)}, \dots, P_{\sigma(n)}) 
\ = \ 
\epsilon(\sigma)^{d_1\cdots d_n}\,
\Res(P_1, \dots, P_n)
$$
\end{prop}

\begin{proof}
Il suffit de faire la preuve en générique. On profite alors du fait que l'idéal d'élimination 
est monogène engendré par le résultant.
Comme $\langle \uP_\sigma \rangle = \langle \uP \rangle$, les idéaux d'élimination sont égaux.
Ainsi, il existe un inversible $u \in \bfA$ tel que $\Res(\uP_\sigma) = u\,\Res(\uP)$.
Pour des raisons de poids, on a $u \in \bfk$.
On peut alors spécialiser sans affecter~$u$. 
Spécialisons en le jeu étalon $\uX^D = (X_1^{d_1}, \dots, X_n^{d_n})$ de résultant $1$.
On a alors 
$$ 
u 
\ = \ 
\Res(X_{\sigma(1)}^{d_{\sigma(1)}}, \dots, X_{\sigma(n)}^{d_{\sigma(n)}}) 
$$
L'utilisation répétée de la multiplicativité du résultant fournit :
$$
u
\ = \ 
\Res\big(X_{\sigma(1)}^{d_{\sigma(1)}}, \dots, X_{\sigma(n)}^{d_{\sigma(n)}}\big)
\ =\  
\Res(X_{\sigma(1)}, \dots, X_{\sigma(n)})^{d_{\sigma(1)} \cdots {d_{\sigma(n)}}}
$$
Or le résultant de $n$ formes linéaires est leur déterminant (cf.~\ref{ResultantFormesLineaires}), 
donc $\Res(X_{\sigma(1)}, \dots, X_{\sigma(n)}) = \epsilon(\sigma)$,
ce qui termine la preuve.
\end{proof}

\subsubsection{Le théorème général de composition pour $\Res(P_1\circ\uQ, \cdots,P_n\circ\uQ)$}

Soit un système $\uQ = (Q_1, \cdots, Q_n)$ où tous les $Q_i$ sont
homogènes de même degré $q$. Pour $F \in \bfA[\uX]$, on pose:
$$
F \circ \uQ = F(Q_1, \cdots, Q_n)
$$
Il est clair que si $F$ est homogène de degré $d$, alors $F \circ \uQ$ est homogène
de degré $qd$.

\label {NOTA12-FoQ}%

On se donne un système $\uP$ de format $D = (d_1, \cdots, d_n)$ et on définit le système
$\uP' = (P'_1, \cdots, P'_n)$ par $P'_i = P_i \circ \uQ$. C'est un système homogène
de format $D' = (d'_1, \cdots, d'_n) := (qd_1, \cdots, qd_n)$. En notant~$\delta'$
son degré critique et $\delta$ celui de $\uP$, on a
$$
\delta' = q\delta + n(q-1)
$$
Le résultat ci-dessous généralise~\ref{LinearComposition} consacré au cas linéaire $q=1$.

\begin {theo} Dans le contexte ci-dessus, on a:
$$
\Res(P_1\circ\uQ, \cdots,P_n\circ\uQ) = \Res(\uQ)^{d_1\cdots d_n}\, \Res(\uP)^{q^{n-1}}
$$
\end {theo}

\begin {proof} \leavevmode

On note $\calR=\Res(\uP)$. Il y a un exposant $e$ tel que $X_i^e\,\calR \in \langle P_1,
\cdots, P_n\rangle$. En évaluant en $X_i := Q_i$:
$$
Q_i^e\, \calR \in \langle P'_1, \cdots, P'_n\rangle
$$
Le théorème de divisibilité fournit alors:
$$
\Res(\uP') \quad \hbox {divise}\quad \Res(Q_1^e\,\calR, \cdots, Q_n^e\,\calR)
$$
D'après le théorème de multiplicativité, le résultant de droite est de
la forme $\calR^s\, \Res(\uQ)^t$ \idest{} $\Res(\uP)^s\, \Res(\uQ)^t$ pour
deux exposants $s,t \ge 1$.

On peut se permettre de passer en générique sur $\bbZ$ pour
$\uP, \uQ$.  On dispose donc d'une relation de divisibilité
$\Res(\uP') \mid \Res(\uP)^s\, \Res(\uQ)^t$ dans un anneau de
polynômes sur $\bbZ$ dans lequel les seuls inversibles sont $\pm 1$ et
$\Res(\uP)$, $\Res(\uQ)$ des irréductibles non associés. On en
déduit qu'il y a des exposants $a, b\ge 0$ et $\varepsilon = \pm 1$
tels que:
$$
\Res(\uP') = \varepsilon\, \Res(\uP)^a\, \Res(\uQ)^b
$$
Pour déterminer $(a,b,\varepsilon)$, on spécialise en prenant $P_i = p_iX_i^{d_i}$ et $Q_j = q_jX_j^q$ où
$p_1, \cdots, p_n, q_1, \cdots, q_n$ sont des indéterminées sur~$\bbZ$.
Alors $P'_i = p_i(q_i X_i^q)^{d_i} = p_iq_i^{d_i} X_i^{d'_i}$. D'où:
$$
\Res(\uP') = (p_1q_1^{d_1})^{\widehat {d'_1}} \cdots  (p_nq_n^{d_n})^{\widehat{d'_n}}
$$
Mais:
$$
\widehat{d'_i} = q^{n-1} \widehat {d_i} \qquad \text {donc} \qquad
d_i \widehat{d'_i} = d_1 \cdots d_n\, q^{n-1}
$$
ce qui conduit à
$$
\Res(\uP') = \big( p_1^{\widehat {d_1}} \cdots p_n^{\widehat {d_n}} \big)^{q^{n-1}}
\times \big( (q_1 \cdots q_n)^{q^{n-1}}\big)^{d_1 \cdots d_n}
\qquad \hbox {ou encore} \qquad
\Res(\uP')  =  \Res(\uP)^{q^{n-1}}\, \Res(\uQ)^{d_1\cdots d_n}
$$
On obtient alors les exposants $a = q^{n-1}$, $b = d_1\cdots d_n$ ainsi que $\varepsilon = 1$,
ce qui termine la preuve du théorème.
\end {proof}

\subsubsection{Le cas particulier de la composition $P'_i = P_i(X_1^q, \cdots, X_n^q)$}
\label{CasParticulierComposition}

Le corollaire qui suit apporte un certain nombre de précisions dans le
cas particulier $Q_i = X_i^q$. Nous suggérons cependant au lecteur
d'étudier au préalable le cas où $\uP$ est linéaire, qu'ici nous
notons plutôt~$\uL$, de sorte que $\uP'$, défini par $P'_i =
L_i(X_1^q, \cdots, X_n^q)$, est de format $D' = (q, \cdots, q)$, de
degré critique $\delta' = n(q-1)$.  Dans ce cas particulier, en ayant
introduit la forme coordonnée sur le $D'$-mouton-noir:
$$
\mu = (X_1^{q-1} \cdots X_n^{q-1})^\star :
\bfA[\uX]_{\delta'} \to \bfA
$$
le résultat principal est le suivant
$$
\omegaRes{\uP'} = \lambda\, \mu  \qquad \text{avec} \qquad
\lambda = \det(\dsL)^{q^{n-1}-1}
$$
Pour le montrer, on pourra utiliser le fait que $\mu$ est une base du
sous-module de $\bfA[\uX]_{\delta'}^\star$ constitué des formes
linéaires nulles sur $\Jex_{1,\delta'}(D') = \langle X_1^q, \dots,
X_n^q\rangle_{\delta'}$. D'autre part, on dispose de la double inclusion suivante:
$$
\det(\dsL)\,\Jex_{1,\delta'}(D') \quad\subseteq\quad \langle \uP'\rangle_{\delta'}
\quad\subseteq\quad  \Jex_{1,\delta'}(D')
$$
Elle montre, en supposant $\uL$ générique, que $\omegaRes{\uP'}$ appartient
à ce sous-module de $\bfA[\uX]_{\delta'}^\star$, donc est multiple de $\mu$.
Le multiplicateur se détermine en évaluant les formes linéaires en le
$D'$-mouton-noir.

\smallskip
Dans le cas très particulier où $n=3$, $q=2$, $D' = (2,2,2)$ et $\delta'=3$,
le résultat principal peut être établi encore plus directement. Adoptons
les notations suivantes:
$$
[L_1,L_2,L_3] = [X_1,X_2,X_3]\,\begin {bmatrix}
a_1 &a_2 &a_3 \\
b_1 &b_2 &b_3 \\
c_1 &c_2 &c_3 \\    
\end {bmatrix}
\qquad\qquad
\setlength{\tabcolsep}{2pt}
\left\{
\begin{tabular}{rcp{5cm}} 
$P'_{1}$ & $=$ & $a_{1}X_{1}^{2} + b_{1}X_{2}^{2} + c_{1}X_{3}^{2}$\\ [0.1cm] 
$P'_{2}$ & $=$ & $a_{2}X_{1}^{2} + b_{2}X_{2}^{2} + c_{2}X_{3}^{2}$\\ [0.1cm] 
$P'_{3}$ & $=$ & $a_{3}X_{1}^{2} + b_{3}X_{2}^{2} + c_{3}X_{3}^{2}$\\ [0.1cm] 
\end{tabular}
\right.
$$
Vu le format $D' = (2,2,2)$, on a $\omegaRes{\uP'} = \omega_{\uP'}$. On
va fournir $\Syl_{\delta'}(\uP')$  dans des bases ad-hoc de
$\rmK_{0,\delta'}$ (de dimension 10) et $\rmK_{1,\delta'}$ (de dimension 9):
$$
\rmK_{1,\delta'} = \bigoplus_i \bfA[\uX]_{\delta'-q}\,e_i =
\bfA[\uX]_1\,e_1 \oplus \bfA[\uX]_1\,e_2 \oplus \bfA[\uX]_1\,e_3 
$$
La voici:
$$
\Syl_{\delta'}(\uP') =
\NorthEastBordermatrix{
\Veti{X_{1}\,e_{1}} & \Veti{X_{1}\,e_{2}} & \Veti{X_{1}\,e_{3}} & \Veti{X_{2}\,e_{1}} & \Veti{X_{2}\,e_{2}} & \Veti{X_{2}\,e_{3}} & \Veti{X_{3}\,e_{1}} & \Veti{X_{3}\,e_{2}} & \Veti{X_{3}\,e_{3}} & \\
a_{1} & a_{2} & a_{3} & . & . & . & . & . & . & \Heti{X_{1}^{3}} \\
b_{1} & b_{2} & b_{3} & . & . & . & . & . & . & \Heti{X_{1}X_{2}^{2}} \\
c_{1} & c_{2} & c_{3} & . & . & . & . & . & . & \Heti{X_{1}X_{3}^{2}} \\
. & . & . & a_{1} & a_{2} & a_{3} & . & . & . & \Heti{X_{1}^{2}X_{2}} \\
. & . & . & b_{1} & b_{2} & b_{3} & . & . & . & \Heti{X_{2}^{3}} \\
. & . & . & c_{1} & c_{2} & c_{3} & . & . & . & \Heti{X_{2}X_{3}^{2}} \\
. & . & . & . & . & . & a_{1} & a_{2} & a_{3} & \Heti{X_{1}^{2}X_{3}} \\
. & . & . & . & . & . & b_{1} & b_{2} & b_{3} & \Heti{X_{2}^{2}X_{3}} \\
. & . & . & . & . & . & c_{1} & c_{2} & c_{3} & \Heti{X_{3}^{3}} \\
. & . & . & . & . & . & . & . & . & \Heti{X_{1}X_{2}X_{3}} \\
}
$$
La dernière ligne mouton-noir $X_1X_2X_3$ étant nulle (pourquoi?), tous les mineurs d'ordre 9 sont
nuls à l'exception de l'un d'entre eux. Ceci a pour conséquence que la forme linéaire $\omega_{\uP'}$
des mineurs signés de $\Syl_{\delta'}(\uP')$ est \og portée\fg{} par $\mu = (X_1X_2X_3)^\star$ et
que
$$
\omegaRes{\uP'} = \omega_{\uP'} = \det(\dsL)^3\, \mu
$$
Après ces quelques lignes consacrées au cas particulier $\uP = \uL$
linéaire dont l'étude complète est laissée au lecteur, voici l'énoncé
plus général relatif à la composition $P_i(X_1^q, \cdots, X_n^q)$.

\begin {coro}
Pour $F \in \bfA[\uX]$, on note $F(\uX^q)$ pour $F(X_1^q,\cdots,X_n^q)$ et on définit
l'injection monomiale
$$
\theta_q : \bfA[\uX] \mapsto \bfA[\uX], \qquad
F \mapsto (X_1\cdots X_n)^{q-1}\, F(\uX^q)
$$
Soit $\uP$ un système de format $D = (d_1, \cdots, d_n)$ et $\uP' =
(P'_1, \cdots, P'_n)$ le système de format $D' = (qd_1, \cdots, qd_n)$
défini par $P'_i=P_i(\uX^q)$.  On note $\delta$ le degré critique de
$\uP$ et $\delta' = q\delta + n(q-1)$ celui de~$\uP'$.

\begin {enumerate} [\rm i)]
\item
Les résultants de $\uP$ et $\uP'$ sont reliés par l'égalité:
$$
\Res(\uP') = \Res(\uP)^{q^{n-1}}
$$

\item
L'image de $\bfA[\uX]_\delta$ par $\theta_q$ est contenue dans $\bfA[\uX]_{\delta'}$
et $\theta_q$ transforme le mouton noir $X^\emouton$ en le mouton noir $X^{D' - \Un}$.
De plus, l'image par $\theta_q$ de tout déterminant bezoutien de $\uP$ est un déterminant
bezoutien de $\uP'$.

\item
Les formes fondamentales de $\uP$ et $\uP'$ sont reliées par l'égalité sur $\bfA[\uX]_\delta$:
$$
\omegaRes{\uP'} \circ \theta_q = \lambda \times \omegaRes{\uP}  \qquad \text{avec}\qquad
\lambda = \Res(\uP)^{q^{n-1}-1}   
$$

\item
La forme $\omegaRes{\uP'}$ est nulle sur le supplémentaire monomial de $\theta_q(\bfA[\uX]_\delta)$
dans $\bfA[\uX]_{\delta'}$ i.e. sur les $(X^{\beta})_{|\beta|=\delta'}$ tels qu'il existe un $j$ vérifiant
$\beta_j \not\equiv -1 \bmod q$.

En conséquence, en notant $\pi_q : \bfA[\uX]_{\delta'} \to
\theta_q(\bfA[\uX]_\delta)$ la projection parallèlement à ce
supplémentaire monomial, on a l'égalité sur $\bfA[\uX]_{\delta'}$:
$$
\omegaRes{\uP'} = \lambda \times (\omegaRes{\uP} \circ \theta_q^{-1} \circ \pi_q)
$$
\end {enumerate}
\end {coro}

\label {NOTA12-FXq}%
%
%

\begin {proof} \leavevmode

L'application $\theta_q$ est bien monomiale et injective puisque:
$$
\theta_q(X^\alpha) = X^{\alpha'}    \quad \text{avec} \quad
\alpha' = q\alpha + (q-1, \cdots, q-1)
$$
De plus $|\alpha'| = q|\alpha| + n(q-1)$. Et l'on a $\alpha'_i 
\equiv -1 \bmod q$ pour tout $i$.  Il est facile de voir que
$\theta_q(\bfA[\uX]_\delta)$ est le sous-module monomial de
$\bfA[\uX]_{\delta'}$ de base les $(X^\beta)_{|\beta|=\delta'}$ tels que $\beta_i 
\equiv -1 \bmod q$ pour tout $i$.

\medskip  
i) Résulte directement du théorème précédent.

\medskip
ii) La première partie ne pose pas de difficulté. Soit $\nabla = \nabla(\uX)$ un déterminant bezoutien
de $\uP$ donc $\nabla = \det \dsV(\uX)$ avec une matrice homogène $\dsV(\uX)$ vérifiant:
$$
[\uP] = [\uX].\dsV(\uX)
$$
En réalisant la substitution $X_i := X_i^q$, on obtient une matrice
bezoutienne $\dsV'(\uX)$ de $\uP'$:
$$
[\uP'] = [\uX^q]. \dsV(\uX^q) = [\uX]. \dsV'(\uX)
\qquad \text{avec} \qquad
\dsV'(\uX) = \diag(X_1^{q-1}, \cdots, X_n^{q-1}) \dsV(\uX^q)
$$
dont le déterminant est $(X_1\cdots X_n)^{q-1}\, \nabla(\uX^q)$, égal à
$\theta_q(\nabla)$ par définition.

\medskip
iii)
Par définition, $\theta_q(P_i) = (X_1\cdots X_n)^{q-1} P'_i \in \langle \uP'\rangle$.
Donc $\omegaRes{\uP'} \circ \theta_q$ est nulle sur $\langle \uP\rangle$. On
peut supposer $\uP$ générique donc il existe $\lambda$ dans l'anneau des
coefficients de $\uP$ tel que:
$$
\omegaRes{\uP'} \circ \theta_q = \lambda \times \omegaRes{\uP}
$$
En évaluant en un déterminant bezoutien $\nabla_\uP$ et en utilisant
que $\theta_q(\nabla_\uP)$ est un déterminant bezoutien de $\uP'$, il
vient:
$$
\Res(\uP') = \lambda\, \Res(\uP)
$$
d'où la détermination de $\lambda$ en utilisant le point i).

\medskip
iv) Pour un monôme $X^\beta$ de degré $\delta'$  tel qu'il existe
un $j$ vérifiant $\beta_j \not\equiv -1 \bmod q$, montrons que
$$  
\Res(\uP)\, X^\beta \in \langle \uP'\rangle_{\delta'}
\leqno (\star)
$$  
On introduit $\alpha\in \bbN^n$ en posant $\alpha_i = \lfloor \beta_i/q\rfloor$. D'après
le lemme qui suit, on a $|\alpha| \ge \delta + 1$. Donc, d'après le théorème~\ref{omegaresNablaInElimIdeal}
$$
\Res(\uP)\, X^\alpha \in \langle \uP\rangle
$$
En réalisant la substitution $X_i := X_i^q$ dans cette appartenance, il vient:
$$
\Res(\uP)\, X^{q\alpha} \in \langle \uP'\rangle
$$
Ceci prouve $(\star)$ puisque $X^\beta$ est multiple de $X^{q\alpha}$ vu que, par définition
de $\alpha$, on a $\beta_i \ge q\alpha_i$ pour tout $i$.

\smallskip 

Pour montrer le point iv) on peut supposer $\uP$ générique auquel cas $\Res(\uP)$ est
régulier. Comme $\Res(\uP)\, X^\beta \in \langle \uP'\rangle_{\delta'}$, on a
$$
\Res(\uP)\, \omegaRes{\uP'}(X^\beta) = 0 \qquad \text{donc} \qquad
\omegaRes{\uP'}(X^\beta) = 0
$$
\end {proof}

\begin {lem} 
Soient $\delta \in \bbN$, $q \in \bbN^*$ et $\beta \in \bbN^n$ tels que $|\beta| = q\delta + n(q-1)$.
On suppose qu'il existe un indice $j$ tel que $\beta_j \not\equiv -1 \bmod q$. Alors:
$$
  \sum_{i=1}^n \Big\lfloor\frac{\beta_i}{q}\Big\rfloor \ge \delta + 1
$$  
\end {lem}   

\begin {proof} 
Ecrivons la division euclidienne de $\beta_i$ par $q$
$$
\beta_i = q q_i + r_i  \quad 0 \le r_i \le q-1  \qquad \text {de sorte que} \qquad
q_i = \Big\lfloor\frac{\beta_i}{q}\Big\rfloor
$$  
On a alors
$$
\sum_i \beta_i = q\sum_i q_i + \sum_i r_i  \quad \text{\idest} \quad
q\delta + n(q-1) =  q\sum_i q_i + \sum_i r_i
$$
Ou encore:
$$
q \Big(\sum_i q_i - \delta\Big) = \sum_i (q-1-r_i)
$$
Chaque $q-1-r_i$ est $\ge 0$ et $q-1 -r_j > 0$ puisque $\beta_j \not\equiv -1 \bmod q$.
On en déduit que la somme de droite est $> 0$. Il en est donc de même de celle de gauche,
ce qu'il fallait montrer.
\end {proof}  

\cleardoublepage

\section{Des identités entre déterminants excédentaires $\det W_\calM(\protect\uP)$} 
\label{ChapWW}

Concernant les déterminants $\det W_\calM(\uP)$, commençons par
rappeler nos hésitations terminologiques évoquées dans la dernière
section du chapitre \ref{ChapObjetsSylvester}: excédentaires ou de
Macaulay? Comme nous n'avons pas vraiment tranché nous utiliserons
indifféremment l'un ou l'autre. 

\medskip

Ce chapitre est consacré à l'étude de diverses relations entre
déterminants de Macaulay $\det W_\calM(\uP)$.  Nous en avons déjà
rencontré dans la section \ref{SousSectionDecompositionJ1circJ1plus}
où nous avons défini une décomposition $\calM = \calM^+ \oplus
\calM^\circ$ et montré:
$$
\det W_\calM(\uP) \ = \ \det W_{\calM^+/X_n}(\uP) \times \det W_{\calM^\circ}(\uP)
\qquad 
\text{où } \calM \subset \Jex_{1,d+1} \text{ et }  d \geqslant \delta+1
$$
Ce qui a conduit au résultat suivant (cf. la proposition~\ref{XnDecompositionJ1d})
$$
\det W_{1,d+1}(\uP) \ = \ \det W_{1,d}(\uP) \times \det W^\circ_{1,d+1}(\uP)
\qquad d \geqslant \delta + 1
$$
Une telle égalité prouve que $\det W_{1,d}(\uP)$ divise $\det W_{1,d+1}(\uP)$
pour $d \geqslant \delta+1$ et explicite le quotient comme un déterminant de Macaulay.
Il est important de noter que la relation ci-dessus peut être écrite de
la manière suivante en faisant intervenir le $(n-1)$-système $\uP'$ de
$\bfA[X_1, \dots, X_{n-1}]$ (cf l'énoncé du théorème \ref{DetW1dDiviseDetW1dplus1})
$$
\det W_{1,d+1}(\uP) \ =\ \det W_{1,d}(\uP) \times \det W_{1,d+1}(\uP')
\qquad d \geqslant \delta + 1
$$
Pour réaliser cette étude, il est indispensable d'étudier de manière
très précise, non pas les \textit{déterminants} --- par on ne sait quel  calcul
\og brute force \fg{} ne menant probablement à rien --- mais les 
\textit{endomorphismes} dont ils sont issus. 
Et de ce fait, les déterminants seront rarement comparés directement. 
L'essentiel de notre tâche va plutôt consister à faire
intervenir des décompositions monomiales, des induit-projetés etc.
Nous n'avons pas compté les décompositions triangulaires qui interviennent
dans la suite ni les petits diagrammes commutatifs qui nous ont permis de
venir à bout des relations déterminantales, mais il y en a un certain nombre !

\medskip

Une des activités principales va ainsi tourner autour du partitionnement des
monômes en catégories adéquates ou de manière équivalente autour de
décompositions ad-hoc de sous-modules monomiaux. Ceci permettra par exemple
d'obtenir, pour $h \geqslant 2$ fixé, la \og relation de divisibilité en $d$ \fg:
$$
\forall\, d \in \bbN, \quad 
\det W_{h, d+1} \ = \ \det W_{h,d} \times \det W_{h,d+1}^\eq
$$
Un autre aspect de l'étude est l'obtention de relations de récurrence entre
divers déterminants de Macaulay conduisant à la détermination de $\Res(\uP)$ et
de la forme $\omegaRes{\uP}$.

\smallskip

Nous étudierons également la \og relation de divisibilité en $h$\fg{}
qui est plus subtile car elle repose à la base sur $\det W_{2,d}(\uP)
\mid \det W_{1,d}(\uP)$ pour n'importe quel $d$.  Cette dernière relation
de divisibilité n'est abordée que partiellement dans ce chapitre et
fera l'objet de chapitres ultérieurs.  Elle n'est pas de même nature
que les autres au sens où le quotient n'est pas un déterminant de
Macaulay: pour $d \geqslant \delta + 1$, le quotient est le résultant $\Res(\uP)$
et pour $d = \delta$, c'est $\omegaRes{\uP}(X^\emouton)$.

\subsection{\'Egalités déterminantales : stratégies triangulaires à venir} 
\label{subsectionDecompositionTriangulaire}

\subsubsection{Décomposition triangulaire d'un endomorphisme et endomorphismes quasi-conjugués}

\begin{defn} \label{DecompositionTriangulaireDef}
Soit $u$ un endomorphisme d'un module libre $E$ de rang fini.
Une décomposition triangulaire de $u$ est la donnée d'une famille finie 
de sous-modules libres $(E_i)_{i \in I}$ où $(I, \preccurlyeq)$ est un ensemble ordonné, 
le tout vérifiant 
$$
E \ = \ \bigoplus_{i \in I} E_i 
\qquad \hbox{et } \qquad 
u(E_i) \, \subset \, \bigoplus_{i' \preccurlyeq i} E_{i'}
$$
Dans ce cas, $\det u = \prod_{i \in I} \det u_i$ 
et $\chi_u = \prod_{i \in I} \chi_{u_i}$, où $u_i$ désigne l'induit-projeté de $u$ sur $E_i$.

\medskip

On dira également que $u$ est triangulaire relativement à $E = \bigoplus_{i\in I} E_i$, 
sous-entendu relativement à la structure d'ordre $\preccurlyeq$ de $I$.
\end{defn}

\begin{rmq}
Il faut noter qu'il est inutile de supposer que $\preccurlyeq$ est une relation d'ordre 
\textit{total}.
Cependant, pour justifier l'égalité déterminantale, 
on peut considèrer une relation d'ordre total $\leqslant$ sur $I$
vérifiant $i' \preccurlyeq i \Rightarrow i' \leqslant i$, 
et utiliser ce petit fait :
\begin{quote}
\textit{
Si $u$ est triangulaire relativement à $E = \bigoplus_{i \in I} E_i$ et $(I, \preccurlyeq)$
alors $u$ est triangulaire relativement à cette même somme directe et $(I, \leqslant)$.
Autrement dit 
$$
u(E_i) \, \subset \, \bigoplus_{i' \preccurlyeq i} E_{i'}
\quad \Longrightarrow  \quad 
u(E_i) \, \subset \, \bigoplus_{i' \leqslant i} E_{i'}
$$
}
\end{quote}
\parbox{0.8\linewidth}{
En ordonnant $(I, \leqslant)$ de manière \textit{croissante},
la matrice de $u$ dans une base adaptée à la décomposition est alors 
triangulaire \textit{supérieure} par blocs. 
Et chaque bloc diagonal est la matrice de $u_i$ dans la base mentionnée de $E_i$. 
D'où l'égalité déterminantale $\det u = \prod_{i \in I} \det u_i$ 
et celle des polynômes caractéristiques $\chi_u = \prod_{i \in I} \chi_{u_i}$.
}
\parbox{0.2\linewidth}{
\begin{tikzpicture}[scale = 0.6]
\clip (-1,-0.5) rectangle (4.3,3.9) ;
\draw (1.5, 3) node[above] {\tiny $E_i$} ;
\draw (3, 1.5) node[right] {\tiny $E_i$} ;
\draw (3, 2.5) node[right] {\tiny $E_{i'}$} ;
\draw (0,2) rectangle (1,3) ;
\draw (2,0) rectangle (3,1) ;
\draw (2,2) rectangle (3,3) ;
\draw[fill, gray!20] (1,2) rectangle (2,3) ;
\draw (1,2) rectangle (2,3) ;
\draw[fill, gray!60] (1,1) rectangle (2,2) ;
\draw (1.5, 1.5) node {\tiny $u_i$} ;
\draw (1,1) rectangle (2,2) ;
\draw (0,0) rectangle (3,3) ;
\end{tikzpicture}
}

Trouver une relation d'ordre total $\leqslant$ qui couvre la relation d'ordre partiel $\preccurlyeq$ 
est toujours possible (prendre pour plus petit élément de $\leqslant$ 
un élément de $I$ sans $\preccurlyeq$-prédécesseur, etc.).
Dans ce cas, on dit que 
l'ordre total $\leqslant$ est une extension linéaire de l'ordre partiel $\preccurlyeq$.
Dans notre contexte, une extension linéaire de $\preccurlyeq$ sera 
évidente à exhiber. Elle sera du type lex, glex, grevlex\dots 
\end{rmq}

\begin{exemple}
Voici un exemple simple relatif à la structure triangulaire de $W_n$. 
Soit $X^\alpha \in \Jex_n$ ; on a $\minDiv(X^\alpha) = 1$ de sorte que
$W_n(X^\alpha) = \pi_{\Jex_n}\Bigl(\frac{X^\alpha}{X_1^{d_1}}\, P_1\Bigr)$.
Notons $X^\gamma 
= X_1^{\gamma_1} \cdots X_n^{\gamma_n}$ un monôme quelconque de $P_1$, 
donc $\gamma_1 + \cdots + \gamma_n = d_1$. 
Alors le monôme
$X^{\alpha'} = \frac{X^\alpha}{X_1^{d_1}}\,X^\gamma$ a pour exposants
$$
\alpha'_1 = \alpha_1 - (d_1-\gamma_1) \, \leqslant \, \alpha_1, \qquad
\alpha'_2 = \alpha_2 + \gamma_2 \, \geqslant \, \alpha_2, \qquad \dots \qquad
\alpha'_n = \alpha_n + \gamma_n \, \geqslant \, \alpha_n
$$
Définissons l'ordre $\preccurlyeq$ par 
$$
\alpha' \preccurlyeq \alpha
\ \iff \ 
\text{$|\alpha'| = |\alpha|$ et $\alpha'  \leqslant_{\times} \alpha$}
$$
où $\leqslant_{\times}$ désigne la structure d'ordre 
\og produit cartésien \fg{} 
$\bbN_{\leqslant} \times \bbN_{\geqslant} \times \cdots \times \bbN_{\geqslant}$. 
Avec cet ordre $\preccurlyeq$, on a~:
$$
W_n(X^\alpha) \ \in 
\sum_{X^{\alpha'} \in \, \Jex_n \atop \alpha' \preccurlyeq \alpha} \bfA \,X^{\alpha'}
$$
Cette structure d'ordre $\preccurlyeq$ est couverte par la structure d'ordre
lexicographique $\leqslant _{\rm lex}$ au sens où on a l'implication
$\alpha' \preccurlyeq \alpha \Rightarrow \alpha'\leqslant _{\rm lex}\alpha$. 
En effet, si $\alpha' \preccurlyeq \alpha$ alors $\alpha' \leqslant_\times \alpha$.
Donc ou bien $\alpha'_1 < \alpha_1$, auquel cas $\alpha' <_{\rm lex} \alpha$,
ou bien $\alpha'_1 = \alpha_1$ auquel cas $\alpha' = \alpha$ 
(car $|\alpha' | = |\alpha |$ et que $\alpha'_i \geqslant \alpha_i$ pour tout $i$).
Par conséquent, si on range les monômes de
$\Jex_{n,d}$ de manière croissante pour l'ordre $\leqslant _{\rm lex}$, 
la matrice de $W_{n,d}$ est triangulaire supérieure.
\end{exemple}

\begin{defn} \label{EndosQuasiConjugues} 
Soient $(u,E)$ et $(v,F)$ deux endomorphismes.
On dit qu'ils sont quasi-conjugués relativement aux décompositions triangulaires 
$E = \bigoplus_i E_i$ et $F = \bigoplus_i F_i$ portant 
sur le même ensemble ordonné $(I, \preccurlyeq)$
lorsque l'induit-projeté $u_i$ est conjugué à l'induit-projeté $v_i$ ;
autrement dit, lorsqu'il existe un isomorphisme $\psi_i : E_i \rightarrow F_i$ 
faisant commuter le diagramme :
$$
\def \Binduitproj{\begin{array}{c}
     \hbox {$u_i$ induit-proj.} \\ \hbox {de $u$} \end {array}}
\def \Winduitproj{\begin{array}{c}
     \hbox {$v_i$ induit-proj.} \\ \hbox {de $v$} \end {array}}
\xymatrix @R = 1.5cm @C = 5cm @M=0.4pc{
E_i \ar[d]|-{\Binduitproj} \ar[r]^{\psi_i}_{\simeq} &  F_i \ar[d]|-{\Winduitproj}
\\
E_i \ar[r]^{\psi_i}_{\simeq} & F_i
\\
}
$$
En particulier, on a $\det u = \det v$ et même $\chi_u = \chi_v$. 
\end{defn}

\index{quasi-conjugués (endomorphismes)}%

\subsubsection{Un cas d'école de structure triangulaire : le Zeroed-system}

Le résultat qui vient permet notamment de simplifier le calcul de $\det W_\calM(\uP)$ 
si $\calM \subset \Jex_{h,d}$ avec $h \geqslant 2$, en particulier celui de $\det W_{h,d}(\uP)$. 
Tout d'abord, l'endomorphisme $W_\calM(\uP)$ ne dépend pas des $h-1$ derniers polynômes~$P_i$ 
que l'on peut donc remplacer par $X_i^{d_i}$. Ceci est banal. 
Mais ce qui est beaucoup moins banal, 
c'est que les $h-1$ dernières variables peuvent être mises à $0$ sans changer le déterminant.
Introduisons le système $\Ph$ obtenu en spécialisant les $h-1$ dernières variables à $0$ et
en remplaçant les $h-1$ derniers polynômes~$P_i$ par~$X_i^{d_i}$, c'est-à-dire :
$$
\Ph
\ = \ 
\big( 
P_1^{(h)}, \dots, P_{n-h+1}^{(h)},\ X_{n-h+2}^{d_{n-h+2}}, \dots, X_n^{d_n}
\big) 
\quad \text{où } 
P_i^{(h)} = P_i(X_1, \dots, X_{n-h+1}, 0, \dots, 0)
$$
\label{AncienJeuPh}
\begin{prop} \label{ZeroedSystem}
Pour tout $\calM \subset \Jex_{h,d}$, les
endomorphismes $W_\calM(\uP)$ et $W_\calM\big(\uP^{(h)}\big)$ sont
quasi-conjugués et a fortiori ont même déterminant. 
\end{prop}

Avant de fournir la preuve, prenons l'exemple du format $D = (3,1,2)$. On a :
\smallskip

{%
\setlength{\tabcolsep}{2pt}

\noindent
$\left\{
\begin{tabular}{rcp{14.5cm}} 
$P_{1}$ & $=$ & $a_{1}X_{1}^{3} + a_{2}X_{1}^{2}X_{2} + a_{3}X_{1}X_{2}^{2} + a_{4}X_{2}^{3} + 
a_{5}X_{1}^{2}X_{3} + a_{6}X_{1}X_{2}X_{3} + a_{7}X_{2}^{2}X_{3} + 
a_{8}X_{1}X_{3}^{2} + a_{9}X_{2}X_{3}^{2} + a_{10}X_{3}^{3}$\\ [0.1cm] 
$P_{2}$ & $=$ & $b_{1}X_{1} + b_{2}X_{2} + b_{3}X_{3}$\\ [0.1cm] 
$P_{3}$ & $=$ & $c_{1}X_{1}^{2} + c_{2}X_{1}X_{2} + c_{3}X_{2}^{2} + c_{4}X_{1}X_{3} + 
c_{5}X_{2}X_{3} + c_{6}X_{3}^{2}$\\ [0.1cm] 
\end{tabular} 
\right.
$

\smallskip

\noindent
et pour $h = 2$ :

\smallskip
\noindent
$\left\{
\begin{tabular}{rcp{14.5cm}} 
$P_{1}^{(h)}$ & $=$ & $a_{1}X_{1}^{3} + a_{2}X_{1}^{2}X_{2} + a_{3}X_{1}X_{2}^{2} + a_{4}X_{2}^{3}$\\ [0.1cm] 
$P_{2}^{(h)}$ & $=$ & $b_{1}X_{1} + b_{2}X_{2}$\\ [0.1cm] 
$P_{3}^{(h)}$ & $=$ & $X_{3}^{2}$\\ [0.1cm] 
\end{tabular} 
\right.
$
}

\smallskip
\noindent
Pour $\calM = \Jex_{h,d}$ avec $h=2$ et $d=6$, 
voici les deux endomorphismes en question, présentés matriciellement relativement 
à la même base (base soigneusement choisie et commentée en~\ref{OrdreMonomialSurNhGrevlex}).

{\small
$$
\renewcommand \Heti[1] {\omit \quad \mbox{\scriptsize$#1$} \hfil}  
W_{h,d}(\uP) \ = \ 
\EastBordermatrix{
a_{1} & . & . & \VR . & . & \VR . & . & . & . & . & \VR . & . & . & . & \VR . & . & \VR . & \Heti{X_{1}^{5}X_{2}} \\ 
a_{2} & a_{1} & . & \VR . & . & \VR . & . & . & . & . & \VR . & . & . & . & \VR . & . & \VR . & \Heti{X_{1}^{4}X_{2}^{2}} \\ 
a_{3} & a_{2} & a_{1} & \VR . & . & \VR . & . & . & . & . & \VR . & . & . & . & \VR . & . & \VR . & \Heti{X_{1}^{3}X_{2}^{3}} \\ 
\HR{21} 
a_{5} & . & . & \VR a_{1} & . & \VR . & . & . & . & . & \VR . & . & . & . & \VR . & . & \VR . & \Heti{X_{1}^{4}X_{2}X_{3}} \\ 
a_{6} & a_{5} & . & \VR a_{2} & a_{1} & \VR . & . & . & . & . & \VR . & . & . & . & \VR . & . & \VR . & \Heti{X_{1}^{3}X_{2}^{2}X_{3}} \\ 
\HR{21} 
. & . & . & \VR . & . & \VR a_{1} & . & . & . & . & \VR . & . & . & . & \VR . & . & \VR . & \Heti{X_{1}^{4}X_{3}^{2}} \\ 
a_{8} & . & . & \VR a_{5} & . & \VR a_{2} & a_{1} & b_{1} & . & . & \VR . & . & . & . & \VR . & . & \VR . & \Heti{X_{1}^{3}X_{2}X_{3}^{2}} \\ 
a_{9} & a_{8} & . & \VR a_{6} & a_{5} & \VR a_{3} & a_{2} & b_{2} & b_{1} & . & \VR . & . & . & . & \VR . & . & \VR . & \Heti{X_{1}^{2}X_{2}^{2}X_{3}^{2}} \\ 
. & a_{9} & a_{8} & \VR a_{7} & a_{6} & \VR a_{4} & a_{3} & . & b_{2} & b_{1} & \VR . & . & . & . & \VR . & . & \VR . & \Heti{X_{1}X_{2}^{3}X_{3}^{2}} \\ 
. & . & a_{9} & \VR . & a_{7} & \VR . & a_{4} & . & . & b_{2} & \VR . & . & . & . & \VR . & . & \VR . & \Heti{X_{2}^{4}X_{3}^{2}} \\ 
\HR{21} 
. & . & . & \VR . & . & \VR a_{5} & . & . & . & . & \VR a_{1} & b_{1} & . & . & \VR . & . & \VR . & \Heti{X_{1}^{3}X_{3}^{3}} \\ 
a_{10} & . & . & \VR a_{8} & . & \VR a_{6} & a_{5} & b_{3} & . & . & \VR a_{2} & b_{2} & b_{1} & . & \VR . & . & \VR . & \Heti{X_{1}^{2}X_{2}X_{3}^{3}} \\ 
. & a_{10} & . & \VR a_{9} & a_{8} & \VR a_{7} & a_{6} & . & b_{3} & . & \VR a_{3} & . & b_{2} & b_{1} & \VR . & . & \VR . & \Heti{X_{1}X_{2}^{2}X_{3}^{3}} \\ 
. & . & a_{10} & \VR . & a_{9} & \VR . & a_{7} & . & . & b_{3} & \VR a_{4} & . & . & b_{2} & \VR . & . & \VR . & \Heti{X_{2}^{3}X_{3}^{3}} \\ 
\HR{21} 
. & . & . & \VR a_{10} & . & \VR a_{9} & a_{8} & . & . & . & \VR a_{6} & . & b_{3} & . & \VR b_{2} & b_{1} & \VR . & \Heti{X_{1}X_{2}X_{3}^{4}} \\ 
. & . & . & \VR . & a_{10} & \VR . & a_{9} & . & . & . & \VR a_{7} & . & . & b_{3} & \VR . & b_{2} & \VR . & \Heti{X_{2}^{2}X_{3}^{4}} \\ 
\HR{21} 
. & . & . & \VR . & . & \VR . & a_{10} & . & . & . & \VR a_{9} & . & . & . & \VR . & b_{3} & \VR b_{2} & \Heti{X_{2}X_{3}^{5}} \\ 
\noalign{\vskip-1pt}
}
$$
$$
\renewcommand \Heti[1] {\omit \quad \mbox{\scriptsize$#1$} \hfil}  
W_{h,d}\big(\uP^{(h)}\big) \ = \ 
\EastBordermatrix{
a_{1} & . & . & \VR . & . & \VR . & . & . & . & . & \VR . & . & . & . & \VR . & . & \VR . & \Heti{X_{1}^{5}X_{2}} \\ 
a_{2} & a_{1} & . & \VR . & . & \VR . & . & . & . & . & \VR . & . & . & . & \VR . & . & \VR . & \Heti{X_{1}^{4}X_{2}^{2}} \\ 
a_{3} & a_{2} & a_{1} & \VR . & . & \VR . & . & . & . & . & \VR . & . & . & . & \VR . & . & \VR . & \Heti{X_{1}^{3}X_{2}^{3}} \\ 
\HR{21} 
. & . & . & \VR a_{1} & . & \VR . & . & . & . & . & \VR . & . & . & . & \VR . & . & \VR . & \Heti{X_{1}^{4}X_{2}X_{3}} \\ 
. & . & . & \VR a_{2} & a_{1} & \VR . & . & . & . & . & \VR . & . & . & . & \VR . & . & \VR . & \Heti{X_{1}^{3}X_{2}^{2}X_{3}} \\ 
\HR{21} 
. & . & . & \VR . & . & \VR a_{1} & . & . & . & . & \VR . & . & . & . & \VR . & . & \VR . & \Heti{X_{1}^{4}X_{3}^{2}} \\ 
. & . & . & \VR . & . & \VR a_{2} & a_{1} & b_{1} & . & . & \VR . & . & . & . & \VR . & . & \VR . & \Heti{X_{1}^{3}X_{2}X_{3}^{2}} \\ 
. & . & . & \VR . & . & \VR a_{3} & a_{2} & b_{2} & b_{1} & . & \VR . & . & . & . & \VR . & . & \VR . & \Heti{X_{1}^{2}X_{2}^{2}X_{3}^{2}} \\ 
. & . & . & \VR . & . & \VR a_{4} & a_{3} & . & b_{2} & b_{1} & \VR . & . & . & . & \VR . & . & \VR . & \Heti{X_{1}X_{2}^{3}X_{3}^{2}} \\ 
. & . & . & \VR . & . & \VR . & a_{4} & . & . & b_{2} & \VR . & . & . & . & \VR . & . & \VR . & \Heti{X_{2}^{4}X_{3}^{2}} \\ 
\HR{21} 
. & . & . & \VR . & . & \VR . & . & . & . & . & \VR a_{1} & b_{1} & . & . & \VR . & . & \VR . & \Heti{X_{1}^{3}X_{3}^{3}} \\ 
. & . & . & \VR . & . & \VR . & . & . & . & . & \VR a_{2} & b_{2} & b_{1} & . & \VR . & . & \VR . & \Heti{X_{1}^{2}X_{2}X_{3}^{3}} \\ 
. & . & . & \VR . & . & \VR . & . & . & . & . & \VR a_{3} & . & b_{2} & b_{1} & \VR . & . & \VR . & \Heti{X_{1}X_{2}^{2}X_{3}^{3}} \\ 
. & . & . & \VR . & . & \VR . & . & . & . & . & \VR a_{4} & . & . & b_{2} & \VR . & . & \VR . & \Heti{X_{2}^{3}X_{3}^{3}} \\ 
\HR{21} 
. & . & . & \VR . & . & \VR . & . & . & . & . & \VR . & . & . & . & \VR b_{2} & b_{1} & \VR . & \Heti{X_{1}X_{2}X_{3}^{4}} \\ 
. & . & . & \VR . & . & \VR . & . & . & . & . & \VR . & . & . & . & \VR . & b_{2} & \VR . & \Heti{X_{2}^{2}X_{3}^{4}} \\ 
\HR{21} 
. & . & . & \VR . & . & \VR . & . & . & . & . & \VR . & . & . & . & \VR . & . & \VR b_{2} & \Heti{X_{2}X_{3}^{5}} \\ 
\noalign{\vskip-1pt}
}
$$
}
On constate que les déterminants sont égaux !

\label {NOTA13-hderniers}%
%
%

\begin{proof}[Preuve de~\ref{ZeroedSystem}]
Commençons par donner la décomposition triangulaire 
de l'endomorphisme $W_\calM(\uP)$. 
De manière informelle, le critère de sélection consiste à regrouper les monômes
selon leurs $h-1$ derniers exposants, si bien qu'il va être commode
d'introduire la suite $\hderniers = (n-h+2, \ldots, n)$ des $h-1$ derniers éléments de $\{1..n\}$. 
Formellement, pour $L \in \bbN^\hderniers$,
notons $\calM^{(L)}$ le sous-$\bfA$-module monomial de $\calM$
de base les $X^\alpha$ tels que $(\alpha_j)_{j \in \hderniers} = L$ et
considérons la structure d'ordre partiel $\preccurlyeq$ sur
$\bbN^\hderniers$ définie par $L \preccurlyeq L'$ si $L_j \leqslant L'_j$
pour tout $j \in \hderniers$.

\noindent 
$\rhd$ 
Alors, au niveau de l'endomorphisme, on a :
$$
W_\calM(\uP)\big(\calM^{(L)}\big) \ \subset \ 
\bigoplus_{L' \succcurlyeq L} \calM^{(L')}
$$
En effet, pour $X^\alpha \in \calM^{(L)}$, l'élément $W_\calM(\uP)(X^\alpha)$ est
une combinaison $\bfA$-linéaire de monômes $X^{\alpha'}$ du type
$\dfrac{X^\alpha}{X_i^{d_i}} X^\gamma$ où $i = \minDiv(X^\alpha)$ et
$X^\gamma$ est un monôme de $P_i$.
Comme $X^\alpha \in \calM \subset \Jex_h$, l'entier $i$ n'est pas dans $\hderniers$ ;
ainsi,  pour tout $j \in \hderniers$, 
on a $\alpha'_j = \alpha_j + \gamma_j \geqslant \alpha_j$.
Par conséquent, $X^{\alpha'} \in \calM^{(L')}$ avec $L' \succcurlyeq L$.

\noindent 
$\rhd$ 
Cette décomposition triangulaire étant valide pour tout $\uP$, elle l'est également pour $\uP^{(h)}$. 

\noindent 
$\rhd$ 
Quant à la quasi-conjugaison, elle 
est réduite à sa plus simple expression.
C'est l'identité qui permet de quasi-conjuguer.
Prenons $X^\alpha \in \calM^{(L)}$.
Comme $X^\alpha \in \calM \subset \Jex_h$, 
l'entier $i = \minDiv(X^\alpha)$ n'est pas dans $\hderniers$.
Ainsi, la suite des $h-1$ derniers exposants du monôme ${X^\alpha}/{X_i^{d_i}}$ vaut $L$.

Il s'agit d'assurer au niveau de la flèche pointillée la commutativité du diagramme suivant:


$$
\xymatrix @M=0.6pc @R =1cm @C=1.5cm{
X^\alpha \ar[r]^-{\id} \ar[d]_-{W_{\calM^{(L)}}(\uP)} &
X^\alpha
\ar[d]^-{W_{\calM^{(L)}}\big(\uP^{(h)}\big)} \\
\pi_{\calM^{(L)}} \Big( \dfrac{X^\alpha}{X_i^{d_i}}\, P_i \Big)
\ar@{-->}[r]^-{\id} & 
\pi_{\calM^{(L)}} \Big( \dfrac{X^\alpha}{X_i^{d_i}}\, P_i^{(h)} \Big) \\
}
$$
L'élément en bas à gauche est 
une combinaison linéaire de monômes $X^{\alpha'} \in \calM^{(L)}$ 
du type $X^{\alpha'} = \dfrac{X^\alpha}{X_i^{d_i}}\, X^\gamma$ 
où $X^\gamma$ est un monôme de~$P_i$.
Comme la suite des $h-1$ derniers exposants du monôme ${X^\alpha}/{X_i^{d_i}}$ vaut $L$ 
et $X^{\alpha'} \in \calM^{(L)}$, 
on en déduit que le monôme $X^\gamma$ ne dépend pas des $X_j$ pour $j \in \hderniers$, 
donc est un monôme de $P_i^{(h)}$.
Ainsi, $X^{\alpha'}$ est dans 
$\pi_{\calM^{(L)}} \Big( \dfrac{X^\alpha}{X_i^{d_i}}\, P_i^{(h)} \Big)$.
\end{proof}

\medskip

Conformément à~\ref{DecompositionTriangulaireDef}, nous avons obtenu le bonus suivant.

\begin{prop}
Pour $\calM \subset \Jex_{h,d}$ et $h \geqslant 2$, on a
$$
\det W_\calM(\uP) \ = \ 
\prod_{L \in \bbN^\hderniers} \det W_{\calM^{(L)}} (\uP^{(h)}) \ ;
\quad \text{ en particulier, }  
\det W_{h,d}(\uP) \ = \ 
\prod_{L \in \bbN^\hderniers} \det W_{h,d}^{(L)} (\uP^{(h)})
$$
\end{prop}

\begin{rmq} \label{OrdreMonomialSurNhGrevlex}
En ce qui concerne les relations d'ordre $\leqslant$ sur $\bbN^\hderniers$, 
extensions linéaires de $\preccurlyeq$, c'est-à-dire vérifiant 
$$
L \preccurlyeq L'   \quad\Longrightarrow\quad  L \leqslant L'
$$
\textit{n'importe} quelle structure d'ordre monomiale $\leqslant$ sur $\bbN^\hderniers$
convient. 

\noindent
\parbox{0.8\linewidth}{
Ainsi, si on ordonne les $L$ de manière \textit{croissante pour $\leqslant$}, 
alors la matrice de $W_\calM(\uP)$ dans une base adaptée à la décomposition exhibée est 
triangulaire \textit{inférieure} par blocs. 
En effet, pour $L$ fixé et pour un rangement $\leqslant$-croissant, 
les $L'$ tels que $L' \succcurlyeq L$ viennent après $L$, de sorte que 
l'endomorphisme $W_\calM(\uP)$ a l'allure ci-contre :
}
\parbox{0.3\linewidth}{
\begin{tikzpicture}[scale = 0.6]
\clip (-1,-0.5) rectangle (4,3.7) ;
\draw (1.5, 3) node[above] {\tiny $L$} ;
\draw (3, 1.5) node[right] {\tiny $L$} ;
\draw (3, 0.5) node[right] {\tiny $L'$} ;
\draw (0,2) rectangle (1,3) ;
\draw[fill, gray!60] (1,1) rectangle (2,2) ;
\draw (1,1) rectangle (2,2) ;
\draw (2,0) rectangle (3,1) ;
\draw (0,1) rectangle (1,2) ;
\draw[fill, gray!20] (1,0) rectangle (2,1) ;
\draw (1,0) rectangle (2,1) ;
\draw (0,0) rectangle (3,3) ;
\end{tikzpicture}
}

\bigskip

Pour terminer, commentons la base choisie dans l'exemple de la page précédente.
Les monômes sont rangés de manière \textit{décroissante} relativement à l'ordre
monomial \textsc {GrevLex}. On voit que les (six) blocs, indexés par l'exposant sur~$X_3$ 
(ici $h=2$ et $n=3$) sont tels que la suite
$(X_3^i)_{0 \leqslant i \leqslant 5}$ est \textit{croissante} pour l'unique ordre monomial
sur $\bbN = \bbN^\hderniers$ ; et l'aspect triangulaire inférieur est
conforme au paragraphe précédent. Il y a là une propriété 
générale de \textsc {GrevLex}; en notant $\Rev{\hderniers}$ la suite $\hderniers = (n-h+2, \ldots, n)$
à l'envers, on a :
$$
|\alpha| = |\alpha' | \text{\ \ et \ } X^\alpha >_{\textsc {GrevLex}} X^{\alpha'}
\quad \Longrightarrow \quad
(\alpha_j)_{j \in \Rev{\hderniers}} 
\ \leqslant_{\rm lex} \ 
(\alpha'_j)_{j \in \Rev{\hderniers}}
$$
On comprend ainsi l'avantage qu'il y a d'utiliser un rangement des monômes via \textsc{GrevLex}.
Les $h-1$ derniers exposants se retrouvent directement rangés pour l'ordre 
monomial \textsc{Lex} sur $\bbN^{\Rev{\hderniers}}$.
Et ainsi, il n'y a plus qu'à \textit{délimiter} les blocs constitués par l'égalité
des $h-1$ derniers exposants.
\end{rmq}

\subsection{Décomposition $\maxDiv$ et divisibilité $\det W_{h,d}(\protect\uP)\mid \det W_{h,d+1}(\protect\uP)$
            pour $h \geqslant 2$}
\label{SousSectionMaxDivDecomposition}

Etant donné un sous-module monomial $\calM$ de $\Jex_2$, l'objectif est
de définir deux sous-modules monomiaux $\calM^\ssup, \calM^\eq$ de $\calM$,
supplémentaires l'un de l'autre, et vérifiant $\det W_\calM = \det W_{\calM^\ssup}
\times W_{\calM^\eq}$. Le résultat suivant, qui prouve a fortiori que
$\det W_{h,d}$ divise $\det W_{h+1,d}$, est un cas particulier de cette égalité
déterminantale.

\begin{theo} \label{DivisibiliteEntreW2}
Pour $h \geqslant 2$ et tout $d$, on a 
$\det W_{h, d+1} = \det W_{h,d} \times \det W_{h,d+1}^\eq$.
\end{theo}

\medskip

On utilisera à plusieurs reprises le fait que pour
$X^\alpha \in \calM$, l'indice $i := \minDiv(X^\alpha)$ est
strictement inférieur à $j := \maxDiv(X^\alpha)$, ce qui
est dû à l'inclusion $\calM \subset \Jex_2$.

\label {NOTA13-maxDiv}%

\begin{defn} 
Pour $j \in \bbN$,  on note $\calM\double{j}$ 
le sous-module de $\calM$ de base les $X^\alpha \in \calM$ 
tels que $\maxDiv(X^\alpha) = j$,  de sorte que $\calM = \bigoplus\limits_{j=1}^n \calM\double{j}$, 
égalité appelée {\rm maxDiv-décomposition}.

Introduisons également $\calM^\ssup$ (resp. $\calM^\eq$)
le module de base les monômes $X^\alpha \in \calM$ tels que
$\alpha_j > d_j$ (resp. $\alpha_j = d_j$) où $j = \maxDiv(X^\alpha)$ de sorte que 
$\calM = \calM^\ssup \oplus \calM^\eq$.

Enfin, $W^\ssup_{h,d}$ désigne l'endomorphisme $W_{\calM^\ssup}$ où $\calM = \Jex_{h,d}$.
Idem pour $W^\eq_{h,d}$
\end{defn}

\label {NOTA13-M[[j]]}%
\label {NOTA13-Mssup}%
\label {NOTA13-Meq}%
\label {NOTA13-Whdssup}%
\label {NOTA13-Whdeq}%

Bien que la décomposition 
$\calM =\calM^\ssup \oplus \calM^\eq$ ne soit pas triangulaire pour $W_\calM$, 
on a quand même une belle égalité déterminantale 
(cf.~\ref{maxDivDecompositionTriangulaire}-iii ci-dessous).
En combinant la $(\ssup,\eq)$-décomposition  et la maxDiv-décomposition, on 
obtient la double décomposition suivante :
$$
\calM \ =\ 
\bigoplus_{j=1}^n \Big(\calM\double{j}^\ssup \oplus \calM\double{j}^\eq \Big)
$$
qui est, comme nous allons le voir, \og doublement \fg{} triangulaire pour $W_\calM$.


Pour illustrer ce que nous venons de raconter de manière informelle, 
soit $D=(d_1,d_2,d_3)$ avec $d_1 = d_3$.
$$
P_1 \ = \  p_1 X_1^{d_1} \, +\,  a X_3^{d_3}, \qquad 
P_2 \ = \ p_2 X_2^{d_2} \, +\,  bX_2^{d_2-1}X_3, \qquad
P_3 \ = \ X_3^{d_3}
$$
Prenons $\calM = (X_1^{d_1}X_2^{d_2+1},\, X_2^{d_2+1}X_3^{d_3},\, X_2^{d_2}X_3^{d_3+1})$
de sorte que $\calM^\ssup = 
(X_1^{d_1}X_2^{d_2+1},\, X_2^{d_2}X_3^{d_3+1})$
et 
$\calM^\eq = (X_2^{d_2+1}X_3^{d_3})$.
La décomposition $\calM = \calM^\ssup \oplus \calM^\eq$
\textbf{n'}est \textbf{pas} triangulaire, comme nous pouvons le constater :
\vspace{-0.3cm} 
$$
W_\calM \ = \ 
\EastBordermatrix{
p_1 & . & \VR . & \Heti{X_1^{d_1}X_2^{d_2+1}} \\ 
. & p_2 & \VR b & \Heti{X_2^{d_2}X_3^{d_3+1}} \\ 
\HR{3} 
a & . & \VR p_2 & \Heti{X_2^{d_2+1}X_3^{d_3}} \\ 
\noalign{\vskip-1pt}
}
$$
En revanche, en raffinant la décomposition suivant la valeur du maxDiv 
(matrice ci-dessous à gauche), et  au sein d'une même valeur de maxDiv,
suivant $\ssup$ ou $\eq$ (matrice ci-dessous à droite), 
on obtient :
\vspace{-0.4cm} 
$$
W_\calM \ = \ 
\EastBordermatrix{
p_1 & \VR . & .& \Heti{X_1^{d_1}X_2^{d_2+1}} \\ 
\HR{4}
a & \VR p_2 & . & \Heti{X_2^{d_2+1}X_3^{d_3}} \\ 
. & \VR b & p_2 & \Heti{X_2^{d_2}X_3^{d_3+1}} \\ 
\noalign{\vskip-1pt}
}
\qquad \qquad \qquad 
W_\calM \ = \ 
\EastBordermatrix{
p_1 & \VRepaisse . & \VR .& \Heti{X_1^{d_1}X_2^{d_2+1}} \\ 
\HRepaisse{5}
a & \VRepaisse  p_2 & \VR . & \Heti{X_2^{d_2+1}X_3^{d_3}} \\ 
\HR{4}
. & \VRepaisse b & \VR p_2 & \Heti{X_2^{d_2}X_3^{d_3+1}} \\ 
\noalign{\vskip-1pt}
}
$$

\begin{prop} \label{maxDivDecompositionTriangulaire}
\leavevmode

\begin{enumerate}[\rm i)]
\item 
L'endomorphisme $W_\calM$ est triangulaire relativement à 
$\calM = \bigoplus\limits_{j=1}^n \calM\double{j}$.

Plus précisément, pour $j$ fixé:
$$
W_\calM \big(\calM\double{j} \big) 
\ \subset \ 
\bigoplus_{j' \geqslant j}\, \calM\double{j'}
$$

\item
L'endomorphisme $W_{\calM\double{j}}$ stabilise $\calM\double{j}^\ssup$ 
dans la décomposition  $\calM\double{j}= \calM\double{j}^\ssup \oplus \calM\double{j}^\eq$.

\item 
On a $\det W_\calM =\det W_{\calM^\ssup} \times \det W_{\calM^\eq}$.
\end{enumerate}
\end{prop}

\begin{proof}

i) Soit $X^\alpha \in \calM\double{j}$ \idest{} $j = \maxDiv(X^\alpha)$.
Posons $i = \minDiv(X^\alpha)$. 
L'image $W_\calM(X^\alpha)$ est une combinaison linéaire 
de monômes de $\calM$ du type $X^\beta = (X^\alpha/X_i^{d_i}) X^\gamma$.
Comme $\calM \subset \Jex_2$, on a $i < j$, a fortiori $j \neq i$. Ainsi $\beta_j = \alpha_j + \gamma_j$,
ce qui implique $\beta_j \geqslant \alpha_j$, donc  
${\maxDiv(X^\beta) \geqslant j}$.

\smallskip
ii) Il suffit de reprendre le point précédent avec $X^\alpha \in \calM\double{j}^\ssup$.
L'inégalité $\alpha_j > d_j$ implique $\beta_j > d_i$, donc $X^\beta$ appartient 
à $\calM\double{j}^\ssup$.

\smallskip
iii) D'après i) et ii), on a 
$$
\det W_\calM = \prod\limits_{j=1}^n 
\det W_{\calM\double{j}} 
\qquad \text{ et } \qquad 
\det W_{\calM\double{j}} = \det W_{\calM\double{j}^\ssup} \times \det W_{\calM\double{j}^\eq}
$$
D'où 
$$
\det W_{\calM} \ = \ 
\prod_{j=1}^n \det W_{\calM\double{j}^\ssup} \times 
\prod_{j=1}^n \det W_{\calM\double{j}^\eq} 
$$
En reprenant le point i) avec le module $\calM^\ssup$ 
qui se décompose en $\bigoplus \calM\double{j}^\ssup$ 
(petit jeu de commutation entre $\double{j}$ et $\ssup$), on obtient  
$\det W_{\calM^\ssup} = \prod_{j=1}^n \det W_{\calM\double{j}^\ssup}$.
Idem pour $\calM^\eq$, d'où le résultat.
\end{proof}

Dans l'énoncé suivant, on rappelle que $W^\ssup_{h,d+1}$ désigne l'endomorphisme $W_{\calM^\ssup}$ où $\calM = \Jex_{h,d+1}$.

\begin{prop} \label{QuasiConjugaisonMaxDivDecomposition}
\leavevmode

\begin{enumerate}[\rm i)]
\item 
Relativement aux décompositions
$\calM^\ssup = \bigoplus_j \calM^\ssup\double{j}$ 
et $\calN = \bigoplus_j \calM^\ssup\double{j}/X_j$,
les endomorphismes $W_{\calM^\ssup}$ et $W_{\calN}$ sont quasi-conjugués.

\item Pour $h \geqslant 2$, on a $\det W_{h,d+1}^\ssup = \det W_{h,d}$.
\end{enumerate}
\end{prop}

\index{quasi-conjugués (endomorphismes)}%

\begin{proof}
i) L'application $\Psi$ suivante conserve $\minDiv$:
$$
\Psi : \ 
\begin{array}[t]{rcl}
\calM & \longrightarrow & \bfA[\uX] \\ [0.3cm]
X^\alpha & \longmapsto & \dfrac{X^{\alpha}}{X_{j_\alpha}} 
\quad \text{où $j_\alpha = \maxDiv(X^\alpha)$}
\end{array}
$$
Sa restriction à $\calM^\ssup$ conserve $\maxDiv$ et sa restriction à
$\calM\double{j}$ est injective.  Ainsi, l'application suivante est
une bijection qui conserve $\minDiv$ et $\maxDiv$, dont la réciproque
est la multiplication par $X_j$ :
$$
\Psi_j : \ 
\begin{array}[t]{rcl}
\calM^\ssup\double{j} & \longrightarrow & \dfrac{\calM^\ssup\double{j}}{X_j} \\ [0.3cm]
X^{\alpha} & \longmapsto & \dfrac{X^{\alpha}}{X_j} 
\end{array}
$$
Montrons que pour $j$ fixé, les endomorphismes 
$W_{\calM^\ssup\double{j}}$ et $W_{\frac{\calM^\ssup\double{j}}{X_j}}$ sont 
conjugués par $\Psi_j$.

Soit $X^\alpha \in \calM^\ssup\double{j}$ et $i = \minDiv(X^\alpha)$.
Il s'agit d'assurer au niveau de la flèche pointillée la commutativité du diagramme suivant:
$$
\xymatrix @M=0.6pc @R =1cm @C=1.5cm{
X^\alpha \ar[r]^-{\Psi_j} \ar[d]_-{W_{\calM^\ssup\double{j}}} &
\dfrac{X^\alpha}{X_j}
\ar[d]^-{W_{\frac{\calM^\ssup\double{j}}{X_j}}} \\
\pi_{\calM^\ssup\double{j}} \Big( \dfrac{X^\alpha}{X_i^{d_i}}\, P_i \Big)
\ar@{-->}[r]^-{\Psi_j} & 
\pi_{\frac{\calM^\ssup\double{j}}{X_j}} \Big(\dfrac{X^\alpha}{X_i^{d_i}X_j} \,P_i\Big) \\
}
$$
Cela résulte, pour  un monôme $X^\gamma$ quelconque de $P_i$, de l'équivalence
suivante, obtenue en utilisant la réciproque de $\Psi_j$:
$$
\dfrac{X^\alpha}{X_i^{d_i}} X^\gamma \in  
\calM^\ssup\double{j}
\ \iff \ 
\dfrac{X^\alpha}{X_i^{d_i}X_j} X^\gamma \in 
\frac{\calM^\ssup\double{j}}{X_j}
$$

ii)  Pour $\calM = \Jex_{h,d+1}$, on a $\calN = \Jex_{h,d}$.
D'après i), les endomorphismes 
$W_{h,d+1}^\ssup$  et $W_{h,d}$ sont quasi-conjugués 
relativement à la $\maxDiv$-décomposition et ont donc même déterminant.
\end{proof}

\begin{proof}[Preuve du théorème~\ref{DivisibiliteEntreW2}]
D'après~\ref{maxDivDecompositionTriangulaire}-iii avec $\calM = \Jex_{h,d+1}$, on a 
$$
\det W_{h,d+1} = \det W_{h,d+1}^\ssup \times \det W_{h,d+1}^\eq
$$
Et d'après le point ii) de~\ref{QuasiConjugaisonMaxDivDecomposition},
on a $\det W_{h,d+1}^\ssup = \det W_{h,d}$, d'où le résultat.
\end{proof}

\subsection{Formules de récurrence pour les déterminants de Macaulay}
\label{sousSectionLastPartition}

On fixe un entier $h \geqslant 2$ dans l'intervalle $\{1..n\}$ et on pose
\framebox [1.1\width][c]{$m = n-(h-1)$}.On note
$\bfA[\Xh] = \bfA[X_1, \dots, X_m]$ le sous-anneau de $\bfA[\uX]$
obtenu en supprimant les $h-1$ dernières variables et pour un $n$-système
$\uP = (P_1, \dots, P_n)$, on définit le $m$-système $\Ph$:
$$
\Ph = (P_{1}^{(h)}, \dots, P_m^{(h)}), 
\quad 
\text{où 
$P_{i}^{(h)} = P_i(X_1, \dots, X_m, 0, \dots, 0) \in \bfA[\Xh]$}
$$
Etant donné un déterminant de Macaulay $\det W_\calM(\uP)$, l'objectif
est, sous certaines conditions portant sur le $n$-système $\uP$ ou
sur $\calM$, d'exprimer ce déterminant en fonction de déterminants de
Macaulay attachés à $\Ph$.  Un exemple typique est le résultat suivant, pour
$h=2$, qui sera utilisé dans la section~\ref{sousSectionMacaulayRecurrence}
pour montrer la formule de Macaulay.

\begin{prop}[Identités pour la preuve par récurrence de la formule de Macaulay]
\label{FormuleRecurrenceW1W2}
\leavevmode

Soit $\uP = (P_1,\dots, P_n)$ un $n$-système avec \framebox [1.1\width][c]{$P_n = X_n^{d_n}$}
et $\uP'$ le $(n-1)$-système $(P'_1, \dots, P'_{n-1})$ défini par $P'_i = P_i(X_n := 0)$
pour $1 \leqslant i \leqslant n-1$. Alors pour tout $d \in \bbN$, on dispose des identités:

\begin{enumerate}[\rm i)]
\item
Au niveau $W_1$:
$$
\det W_{1,d}(\uP) 
\ = \ 
\prod_{j=1}^d \det W_{1, j}(\uP')
$$
\item 
Au niveau $W_2$:
$$
\det W_{2,d}(\uP) \ = \quad 
\prod_{j=1}^{d-d_n} \det W_{1,j}(\uP') 
\ \times \!\!
\prod_{j = d-(d_n- 1)}^d \!\!\det W_{2,j}(\uP') 
$$
\end{enumerate}
\end{prop}

\smallskip
La preuve est reportée plus loin en page~\pageref{ProofRecurrenceW1W2}.
Voici quelques commentaires permettant de préciser un certain nombre
de points de cette section.

\smallskip
$\blacktriangleright$
Tout d'abord, dans le cas particulier ci-dessus, la contrainte bien
visible $P_n = X_n^{d_n}$ sur $\uP$ est capitale pour $W_1$ mais
inutile pour $W_2$ puisque $W_2(\uP)$ ne dépend pas de $P_n$ (ou
plutôt ne dépend que du degré de $P_n$).

\smallskip
$\blacktriangleright$
Le second commentaire concerne les notations.  En~\ref{ZeroedSystem},
nous avions surtout mis l'accent sur le fait de mettre à $0$ les $h-1$
dernières indéterminées.  Ici, nous opérons différemment en supprimant
les $h-1$ derniers polynômes et bien sûr, en mettant à 0 les $h-1$
dernières variables dans les $m$ premiers.  L'anneau de polynômes est
changé : il passe de $\bfA[\uX]$ à $\bfA[\Xh]$ ; et le nouveau format
de degrés est $D^{(h)} = (d_1, \dots, d_{m})$.  Le lecteur attentif
aura remarqué le léger changement de notation pour~$\Ph$
avec~\ref{ZeroedSystem}: on a juste tronqué l'ancienne suite $\Ph$ de
la page~\pageref{AncienJeuPh}.

\smallskip
$\blacktriangleright$
A propos des contraintes portant sur $\uP$ ou $\calM$, précisons comment
elles interagissent l'une sur l'autre tout en restant dans $\bfA[\uX]$.
La suite des $h-1$ derniers éléments de $\{1..n\}$ intervenant constamment sera notée~$\hderniers$.
Pour tout système $\uP$, définissons le $n$-système $\uQ$ de la manière suivante:
$$
Q_i = \begin {cases}
P_i      &\text {si } i \notin \hderniers \qquad \text { (au début)}\\
X_i^{d_i} &\text {si } i \in \hderniers \qquad \text { (à la fin)}\\
\end {cases}
$$
Alors, pour tout sous-module monomial $\calM \subset \Jex_{1,d}$, en
notant $\calM'$ le sous-module monomial de $\calM$ de base les
$\{X^\alpha \in \calM \mid \minDiv(X^\alpha) \notin \hderniers\}$, on a
(cf. le lemme~\ref{BiPartitionMinDivLemma}):
$$
\det W_{\calM'}(\uP) = \det W_{\calM}(\uQ)
$$
A gauche de l'égalité, aucune contrainte sur $\uP$ mais une contrainte sur $\calM'$ ;
tandis qu'à droite, c'est l'inverse: pas de contrainte sur $\calM$ (à part celle
d'être contenu dans $\Jex_1$) mais une contrainte sur $\uQ$ concernant les $h-1$
derniers polynômes du système.

\smallskip
$\blacktriangleright$
Ce qui précède nous conduit à faire intervenir 
le module $\bigoplus_{i=1}^{m} \Jex_{1,d}^{(i)}$:
$$
\bigoplus_{i=1}^{m} \Jex_{1,d}^{(i)}  =
\hbox {sous-module monomial de base les $X^\alpha \in \Jex_{1,d}$
tels que $\minDiv(X^\alpha) \notin \hderniers$}
$$
Pour $\calM$ sous-module de ce dernier module, on a alors :
$$
X^\alpha \in \calM  \implies \minDiv(X^\alpha) \notin \hderniers
$$
de sorte que $\calM' = \calM$ avec les notations ci-dessus.
Des exemples de tels modules $\calM$ : ceux vérifiant 
$\calM \subset \Jex_h$, par exemple 
$\calM = \Jex_k$ avec $k \geqslant h$.

Pour un tel $\calM$, nous obtiendrons la formule récurrente suivante (c'est l'objet
de la proposition~\ref{ResultatsSurL}) dont les
ingrédients seront détaillés dans la suite:
$$
\det W_\calM(\uP) \ = \ 
\prod_{\calL \in \bbN^\hderniers} 
\det W_{\calM\crochet{\calL}/X^\calL}(\Ph)
$$
Signalons que les résultats relatifs à  $W_1,W_2$ dans la proposition~\ref{FormuleRecurrenceW1W2}
sont produits de manière uniforme par cette formule via le choix $\calM = 
\bigoplus_{i=1}^{n-1} \Jex_{1,d}^{(i)}$, puis le choix $\calM = \Jex_{2,d}$.

\smallskip
$\blacktriangleright$
Enfin, nous obtiendrons dans un chapitre ultérieur (corollaire~\ref{DetW2diviseDetW1})
la relation de divisibilité:
$$
\forall\, d \in \bbN, \qquad \det W_{2,d} \mid \det W_{1,d}  
$$
La section \ref{sousSectionMacaulayRecurrence} en fournira une justification pour
$d \geqslant \delta$ seulement (cf.~\ref{DetW2diviseDetW1DegreDelta}). 
Alors, grâce aux relations de récurrence mises
en place dans cette section, nous serons capables de prouver la relation $\det W_{h,d} \mid W_{h-1,d}$
pour tout $d$, à condition de disposer de cette relation pour $h=2$.

\subsubsection{La décomposition qualifiée de \og Last-décomposition\fg}

On rappelle que $\hderniers = (m+1, \ldots, n)$ est la suite des $h-1$ derniers éléments de~$\{1..n\}$
où $m = n-(h-1)$ et que nous imposons $\calM \subseteq \bigoplus_{i=1}^{m} \Jex_{1}^{(i)}$
de sorte que $X^\alpha \in \calM  \implies \minDiv(X^\alpha) \notin \hderniers$.

\begin{defns} \leavevmode

\noindent  
On désigne par $\preccurlyeq$ la relation d'ordre sur $\bbN^\hderniers$ composante à composante.
Pour $\calL \in \bbN^\hderniers$, on pose
$$
|\calL| = \sum_\ell \calL_\ell, \qquad\qquad
X^\calL = \prod_{\ell \in \hderniers} X_\ell^{\calL_\ell}
$$
Pour un monôme $X^\alpha \in \bfA[\uX]$, on note $\Last_h(X^\alpha) = (\alpha_{m+1},\dots,\alpha_n)
\in \bbN^\hderniers$. Ainsi le monôme de $\bfA[\Xh]$ :
$$
\dfrac{X^{\alpha}}{X^{\calL_\alpha}}  \quad \text{où $\calL_\alpha = \Last_h(X^\alpha)$}
$$
est le monôme obtenu via la spécialisation $(X_\ell := 1)_{\ell \in \hderniers}$ de $X^\alpha$.

\medskip
\noindent
On désigne par $\Last_h(\calM)$ la partie finie de $\bbN^\hderniers$ définie par
$\Last_h(\calM) = \{\Last_h(X^\alpha) \mid X^\alpha \in \calM\}$.

\medskip
\noindent
Enfin, $\calM\crochet{\calL}$ désigne le sous-module monomial de $\calM$ de base les $X^\alpha$ 
tels que $\Last_h(X^\alpha) = \calL$, de sorte que l'on a la décomposition
qualifiée de {\rm Last-décomposition}:
$$
\calM = \bigoplus\limits_{\calL \in \bbN^\hderniers} \calM\crochet{\calL}
$$
Note : on a $\calM\crochet{\calL} = 0$ pour $\calL \notin \Last_h(\calM)$.
\end{defns}

\label {NOTA13-Lasth}%
\label {NOTA13-LasthM}%
\label {NOTA13-McalL}%

\medskip

\begin{prop} \label{WhTriangulaireEnL}
L'endomorphisme $W_\calM$ est triangulaire relativement à 
$\calM = \bigoplus\limits_{\calL \in \bbN^\hderniers} \calM\crochet{\calL}$ 
au sens où pour $\calL \in \bbN^\hderniers$, on a 
$$
W_\calM \big(\calM\crochet{\calL} \big) 
\ \subset \ 
\bigoplus_{\calL' \succcurlyeq \calL}\, \calM\crochet{\calL'}
$$ 
\end{prop}

\begin{proof}
Soit $X^\alpha \in \calM$ et $i = \minDiv(X^\alpha)$.
L'image $W_\calM(X^\alpha)$ est une combinaison linéaire 
de monômes du type $X^{\alpha'} = (X^\alpha/X_i^{d_i}) X^\gamma$.
On cherche à montrer que $\Last_h(X^{\alpha'}) \succcurlyeq \Last_h(X^{\alpha})$.
Comme $X^\alpha \in \calM$, l'indice $i$ n'est pas dans $\hderniers$,
donc $\Last_h(X^{\alpha'}) = \Last_h(X^\alpha X^\gamma)$, 
qui est clairement $\succcurlyeq \Last_h(X^\alpha)$.
\end{proof}

\medskip
On peut maintenant énoncer le résultat principal de cette section.

\begin{prop} \label{ResultatsSurL}
\leavevmode

\begin{enumerate}[\rm i)]
\item 
Considérons la \textrm{Last-décomposition} 
$\calM =  \bigoplus_{\calL \in \bbN^\hderniers} \calM\crochet{\calL}$ 
de $\calM$ et notons $\calN$ 
le sous-module monomial de $\bfA[\Xh]$ défini par
$\calN = \bigoplus_{\calL \in \bbN^\hderniers} \calM\crochet{\calL}/X^\calL$.

Relativement aux décompositions précédentes, 
les endomorphismes $W_{\calM}(\uP)$ et $W_{\calN}(\Ph)$ sont 
quasi-conjugués.

\item 
On a l'égalité déterminantale 
$$
\det W_\calM(\uP) \ = \  \prod_{\calL \in \bbN^\hderniers}  \det W_{\calM\crochet{\calL}/X^\calL}(\Ph)
\ = \  \prod_{\calL \in \Last_h(\calM)}  \det W_{\calM\crochet{\calL}/X^\calL}(\Ph)
$$

\item 
Pour $k \geqslant h \geqslant 2$, on a, en notant $\calL^+ = \{i \in \hderniers \ | \  \calL_i\geqslant d_i\}$:
$$
\det W_{k,d}(\uP) \ = \ \prod_{\calL \in \bbN^\hderniers} \det W_{k-\#\calL^+,\, d-|\calL|}(\Ph) \ = \
\prod_{\calL \in \Last_h(\Jex_{k,d})} \det W_{k-\#\calL^+,\, d-|\calL|}(\Ph)
$$
\end{enumerate}
\end{prop}

\medskip
Avant la preuve, un commentaire concernant les ensembles qui indexent les produits dans iii).
Dans l'anneau $\bfA[\Xh]$, intervient $\Jex_{k-\#\calL^+,\, d-|\calL|}$ avec $k-\#\calL^+ \geqslant 1$ (puisque
$\#\calL^+ \leqslant h-1$ et $k \geqslant h$).
Supposons $\Jex_{k-\#\calL^+,\, d-|\calL|}$ non réduit à $\{0\}$.
Alors il contient un monôme $X^\beta$ de $\bfA[\Xh]$ et on vérifie facilement que
le monôme $X^\alpha := X^\beta X^\calL$ de $\bfA[\uX]$ appartient à $\Jex_{k,d}$;
en conséquence, $\calL \in \Last_h(\Jex_{k,d})$. Pris à rebours, si
$\calL \in \bbN^\hderniers \setminus \Last_h(\Jex_{k,d})$, alors 
$\Jex_{k-\#\calL^+,\, d-|\calL|} = \{0\}$ et le déterminant correspondant vaut 1.
    
\begin{proof}
i) L'application $\Psi$ suivante conserve $\minDiv$:
$$
\Psi : \ 
\begin{array}[t]{rcl}
\calM & \longrightarrow & \bfA[\Xh]\\ [0.2cm]
X^{\alpha} & \longmapsto & \dfrac{X^{\alpha}}{X^{\calL_\alpha}} 
\quad \text{où $\calL_\alpha = \Last_h(X^\alpha)$}
\end{array}
$$
Sa restriction à chaque $\calM\crochet{\calL}$ est injective.  Ainsi,
l'application suivante est une bijection qui conserve $\minDiv$, et
dont la réciproque est la multiplication par $X^{\calL}$ :
$$
\Psi_\calL : \ 
\begin{array}[t]{rcl}
\calM\crochet{\calL} & \longrightarrow & \calM\crochet{\calL}/X^\calL 
\\ [0.2cm]
X^{\alpha} & \longmapsto & \dfrac{X^{\alpha}}{X^{\calL}} 
\end{array}
$$
Montrons que pour $\calL$ fixé, les endomorphismes 
$W_{\calM\crochet{\calL}}(\uP)$ et $W_{\calM\crochet{\calL}/X^\calL}(\Ph)$ sont 
conjugués par $\Psi_\calL$.

\noindent
Soit $X^\alpha \in  \calM\crochet{\calL}$ et $i = \minDiv(X^\alpha)$.
Comme $X^\alpha \in \calM \subset \Jex_h$, 
l'entier $i = \minDiv(X^\alpha)$ n'est pas dans $\hderniers$.
Ainsi, $\Last_h\Big(\dfrac{X^\alpha}{X_i^{d_i}}\Big) = \calL$.

Il s'agit d'assurer au niveau de la flèche pointillée la commutativité du diagramme suivant:
$$
\xymatrix @R = 1.2cm @C = 3cm @M=0.4pc{
X^{\alpha} \ar[d]_{W_{\calM\crochet{\calL}}(\uP)} \ar[r]^{\Psi_\calL} 
& \dfrac{X^{\alpha}}{X^\calL}  \ar[d]^{W_{\calM\crochet{\calL}/X^\calL}(\Ph)}
\\
\pi_{\calM\crochet{\calL}} \Big(\dfrac{X^{\alpha}}{X_i^{d_i}} \,P_i\Big)
\ar@{-->}[r]^-{\Psi_\calL} & 
\pi_{\calM\crochet{\calL}/X^\calL} \Big( \dfrac{X^{\alpha}}{X_i^{d_i}X^{\calL}}\, P_{i}^{(h)} \Big)
\\
}
$$
L'élément en bas à gauche est 
une combinaison linéaire de monômes $X^{\alpha'} \in \calM\crochet{\calL}$ 
du type $X^{\alpha'} = \dfrac{X^\alpha}{X_i^{d_i}}\, X^\gamma$ 
où $X^\gamma$ est un monôme de~$P_i$.
Comme on a $\Last_h\Big(\dfrac{X^\alpha}{X_i^{d_i}}\Big) = \calL$ 
(cf. l'explication avant le diagramme) 
et~\makebox{$X^{\alpha'} \in \calM\crochet{\calL}$},  on en déduit que le monôme 
$X^\gamma$ ne dépend pas des $X_j$ pour $j \in \hderniers$, 
donc est un monôme de $P_i^{(h)}$.
Ainsi, $\dfrac{X^{\alpha'}}{X^\calL}$ est dans 
$\pi_{\calM\crochet{\calL}/X^\calL} \Big( \dfrac{X^{\alpha}}{X_i^{d_i}X^{\calL}}\, P_{i}^{(h)} \Big)$.

\noindent
ii) 
D'après la propriété triangulaire dégagée dans la proposition précédente~\ref{WhTriangulaireEnL}, on a 
$$
\det W_\calM(\uP) \ = \ 
\prod_{\calL \in \bbN^\hderniers} \det W_{\calM\crochet{\calL}}(\uP)
$$
et le point précédent fournit  $\det W_{\calM\crochet{\calL}}(\uP)
= \det W_{\calM\crochet{\calL}/X^\calL}(\Ph)$. D'où la formule.

\smallskip
\noindent
iii) Comme $k \geqslant h$, on a $\Jex_{k,d} \subset \Jex_{h,d} \subset \bigoplus_{i=1}^m \Jex_{1,d}^{(i)}$
et on peut donc appliquer ii) à $\calM = \Jex_{k,d}$.
Pour conclure, il reste à vérifier que pour $\calL \in \bbN^\hderniers$, on a,
dans~$\bfA[\Xh]$, l'égalité $\calM\crochet{\calL}/X^\calL = \Jex_{k-\#\calL^+,\, d-|\calL|}$.
Soit $X^\beta \in \bfA[\Xh]$ appartenant à $\Jex_{k-\#\calL^+,\, d-|\calL|}$; alors
$X^\alpha := X^\beta X^\calL \in \calM\crochet{\calL}$ et $X^\alpha/X^\calL = X^\beta$. Dans
l'autre sens, soit $X^\alpha \in \calM\crochet{\calL}$; alors $X^\alpha$ s'écrit
$X^\beta X^\calL$ et on a $X^\beta \in \Jex_{k-\#\calL^+,\, d-|\calL|}$.
\end{proof}

\subsubsection*{$\bullet$ Le cas particulier $h=2$}

Allégeons les notations avec un prime, au lieu d'une puissance $(h)$  
$$
\begin {array} {l@{\qquad\qquad}l}
P'_i  \overset{\rm def}{=}  P_i^{(h)} = P_i(X_1,\dots, X_{n-1}, 0) &
\uP'  \overset{\rm def}{=} \Ph =  (P'_1, \dots, P'_{n-1})
\\ [3mm]
\bfA[\uX'] \overset{\rm def}{=} \bfA[\Xh] = \bfA[X_1, \dots, X_{n-1}] &
\Jex'_{h} \overset{\rm def}{=}  \Jex_{h} \cap \bfA[\uX']
\\
\end {array}
$$
On a également $\bbN^\hderniers = \bbN^{\{n\}}$ et pour $\calL =
(\ell) \in \bbN^\hderniers$, l'égalité $\calM\crochet{\calL} = \calM^{(\ell)} =
\{X^\alpha \in \calM \mid \alpha_n = \ell\}$.

\bigskip

\begin{proof} [Preuve de la proposition \ref{FormuleRecurrenceW1W2}]
\label{ProofRecurrenceW1W2}
\leavevmode

\medskip
\noindent
i) L'application du lemme \ref {BiPartitionMinDivLemma} à $I= \{n\}$,
au système $\uQ = \uP$ et  au module $\calN = \Jex_{1,d}$ fournit :
$$
\det W_{1,d}(\uP) = \det W_\calM(\uP) \qquad \text{avec} \qquad
\calM = \bigoplus_{i=1}^{n-1} \Jex_{1,d}^{(i)}
$$
D'autre part, 
on peut appliquer la proposition~\ref{ResultatsSurL} à ce
$\calM$ pour lequel $\calM^{(\ell)} / X_n^{\ell} = \Jex'_{1, d-\ell}$.
Le point~ii) nous donne:
$$
\det W_\calM(\uP) 
\quad = \quad
\prod_{\ell \in \bbN} \det W_{1,d-\ell}(\uP')
\quad = \quad
\prod_{j=1}^d \det W_{1, j}(\uP')
$$
Par transitivité, on obtient l'égalité voulue.

\medskip
\noindent
ii) On applique le point iii) de cette même proposition:
$$
\det W_{2,d}(\uP)
\ = \ 
\prod_{\ell \in \bbN} \det W_{2-\ell^{+},\, d-\ell}(\uP')
\qquad 
\text{où }
\ell^+ \ = \ 
\left \{
\begin{array}{ll}
0 & \text{si } \ell < d_n \\
1 & \text{si } \ell \geqslant  d_n \\
\end{array}
\right.
$$
On tire alors la formule ci-dessous :
$$
\det W_{2,d}(\uP) 
\quad = \quad 
\prod_{\ell < d_n} \det W_{2, d-\ell}(\uP') 
\times 
\prod_{\ell \geqslant d_n} \det W_{1, d-\ell}(\uP') 
$$
qui donne la formule annoncée une fois le changement de variable $j= d-\ell$ effectué.
\end{proof}

\begin{theo} \label{NiveauhDivisehmoins1}
Pour tout $d$ et tout $h \geqslant 2$, on a
$$
\det W_{h,d}(\uP) \ \mid \ \det W_{h-1,d}(\uP)
$$
\end{theo}

\begin{proof}
Procédons par récurrence sur $h$, avec l'hypothèse portant sur tout $d$ et tout système $\uP$.

L'initialisation $h = 2$, hautement non triviale, sera traitée en~\ref{DetW2diviseDetW1}.
Pour l'hérédité, on utilise deux fois l'égalité de~\ref{ResultatsSurL}-iii), une fois avec $k=h$
puis avec $k = h+1$ :
$$
\det W_{h,d}(\uP) \ = \ 
\prod_{\calL \in \bbN^{\hderniers}} \det W_{h-\#\calL^+, d-|\calL|}(\Ph),
\qquad\quad 
\det W_{h+1,d}(\uP) \ = \ 
\prod_{\calL \in \bbN^{\hderniers}} \det W_{h+1-\#\calL^+, d-|\calL|}(\Ph)
$$
D'après l'hypothèse de récurrence appliquée au système $\Ph$, on obtient : 
$$
\forall\, \calL, \qquad 
\det W_{h+1-\#\calL^+, d-|\calL|}(\Ph)
\ \mid \ 
\det W_{h-\#\calL^+, d-|\calL|}(\Ph)
$$
Par produit sur les $\calL$, on obtient $\det W_{h+1,d}(\uP) \mid \det W_{h,d}(\uP)$.
\end{proof}

\subsection{Les formules de Macaulay en degré $\ge \delta$: preuve à l'arraché (récurrence sur $n$)}
\label{sousSectionMacaulayRecurrence}

L'objectif ici est d'identifier les cofacteurs $\big(b_d(\uP)\big)_{d \geqslant \delta}$ introduits
en~\ref{CofacteursTordus}, où $\uP=(P_1,\cdots,P_n)$ est un système quelconque de $\bfA[\uX]$:
$$
\omega 
\ = \ 
 b_\delta \,\omegares 
\qquad \text{ et } \qquad 
\forall\,d \geqslant \delta+1,\ 
\det W_{1,d} = \calR_d \, b_d 
$$
Nous allons montrer, par récurrence sur $n$, que $b_d  = \det W_{2,d}$.
Nous utiliserons d'une part l'égalité des $(\calR_d)_{d \ge \delta+1}$ sous la forme
$\Res(\uP) = \calR_d$ et d'autre part la formule de Poisson pour bébés
vérifiée par le résultant et la forme~$\omegares$ lorsque $P_n = X_n^{d_n}$
(corollaire~\ref{WhenPnIsXndn}).

L'égalité $b_d = \det W_{2,d}$ ci-dessus fournit une expression polynomiale
explicite d'une part pour la forme~$\omegares$ et d'autre part pour le
résultant. Il existe également, pour $\sigma \in \fS_n$, une
$\sigma$-version de cet énoncé, à savoir $b^\sigma_d(\uP) = \det
W^\sigma_{2,d}(\uP)$, mais le point le plus important est de se
concentrer sur la preuve du cas $\sigma = \Id_n$, le cas $\sigma$
quelconque s'adaptant aussitôt.

\medskip

Malgré les analogies entre la preuve du cas $d \geqslant \delta + 1$ et celle du cas $d = \delta$,
nous avons choisi d'en faire deux énoncés.

\begin{theo} [Formule de Macaulay en degré $\geqslant \delta+1$] \label {MacaulayParRecurrence} 
Pour $\sigma \in \fS_n$ et $d \geqslant \delta+1$, on a l'égalité:
$$
\det W^\sigma_{1,d} = \det W^\sigma_{2,d}\ \Res(\uP)
$$
Ce qui fournit, pour le résultant, une formule polynomiale à travers plusieurs expressions:
$$
\Res(\uP) = \frac{\det W_{1,d}}{\det W_{2,d}} = \frac{\det W^\sigma_{1,d}}{\det W^\sigma_{2,d}}
$$
\end{theo}

\index{formules de Macaulay}%

Quant au cas du degré $\delta$, il peut s'énoncer sous la forme suivante:

\begin{theo} [En degré $\delta$] \label {MacaulayParRecurrenceDegreeDelta} 
Pour $\sigma \in \fS_n$, on a l'égalité des formes linéaires
$$
\omega^\sigma = \det W^\sigma_{2,\delta}\ \omegares
$$  
De manière équivalente, la forme $\omegares$ se détermine via plusieurs expressions:
$$
\omegares = \frac{\omega}{\det W_{2,\delta}} = \frac{\omega^\sigma}{\det W^\sigma_{2,\delta}}
$$
\end{theo}

Pour prouver ces formules, il suffit de le faire pour la suite générique.
Nous allons procéder par récurrence sur $n$, 
en supposant $\sigma = \Id_n$. Il va être indispensable de mentionner la dépendance en $\uP$ 
des objets, car vont cohabiter divers systèmes déduits de $\uP$.

Une fois n'est pas coutume, disons un mot à propos de l'initialisation $n=1$!
Dans ce cas, la suite~$\uP$ est $(p_1X_1^{d_1})$ et relève ainsi du
jeu étalon généralisé.  On a $\delta = d_1-1$ et $\widehat d_1 = 1$,
d'où $\Res(\uP) = p_1$ et $\omegaRes{\uP} = (X_1^{d_1-1})^\star$.
Du côté des idéaux
excédentaires, $\Jex_2=0$ donc chaque $W_{2,d}$ est la matrice
vide de déterminant $1$.  Concernant $\Jex_1$ : pour $d \geqslant
\delta+1$, on a $\Jex_{1,d} = \bfA X_1^d$ donc $W_{1,d} = (p_1)$, et
pour $d \leqslant \delta$, on a $\Jex_{1,d} = 0$, donc $\det W_{1,d} = 1$.
Ceci étant, les formules annoncées deviennent évidentes à vérifier.

\smallskip

Avant de commencer à proprement parler les preuves, nous rassemblons
les renseignements à notre disposition. Tout d'abord, $b_d(\uP)$ ne dépend pas
de $P_n$ (cf.~proposition~\ref{ProprietesCofacteurs}), si bien que nous
pouvons supposer \framebox [1.1\width][c]{$P_n = X_n^{d_n}$}. Nous
pouvons également supposer que $P_1, \dots, P_{n-1}$ sont génériques.
Nous notons $\uP' = (P'_1, \dots, P'_{n-1})$ où $P'_i = P_i(X_n := 0)$,
suite dont le degré critique vaut $\delta' = \delta - (d_n - 1)$.

\smallskip
\noindent
$\bullet$ Du côté des déterminants $w_{1,d} := \det W_{1,d}(\uP)$ et $w_{2,d} := \det W_{2,d}(\uP)$,
nous avons, avec des notations analogues pour $\uP'$, les formules de récurrence
(cf. proposition~\ref{FormuleRecurrenceW1W2}) pour tout $d \in \bbN$:
$$
w_{1,d} = \prod_{j=1}^d w'_{1,j} \qquad\qquad
w_{2,d} = \Bigg(\prod_{i=1}^{d-d_n} w'_{1,i}\Bigg)  \ \Bigg(\prod_{j=d-(d_n-1)}^{d} w'_{2,j}\Bigg)
$$
Les déterminants $w_{1,d}$ et $w_{2,d}$ sont des scalaires
réguliers (car primitifs par valeur : spécialiser en le jeu étalon).
Idem pour $\uP'$. Nous pouvons donc nous permettre d'écrire des quotients:
$$
\frac{w_{1,d}}{w_{2,d}} = \prod_{j=d-(d_n-1)}^{d} \frac{w'_{1,j}} {w'_{2,j}}
$$
$\bullet$ Du côté du résultant et de $\omegares$, il est commode d'introduire
une notation locale $\calQ_\delta(\uP) := \omegaRes{\uP}(X^\emouton)$
et son analogue pour $\uP'$. Nous avons alors  (corollaire~\ref{WhenPnIsXndn})
$$
\Res(\uP) = \Res(\uP')^{d_n}
\qquad\text{ et } \qquad
\calQ_\delta(\uP) = \calQ_{\delta'}(\uP')\, \Res(\uP')^{d_n-1}
$$

\begin {proof} [\bf Preuve de la formule en degré $\geqslant \delta+1$.]
Occupons-nous de l'hérédité de la preuve par récurrence sur~$n$.
Fixons $d \geqslant \delta + 1$.
Il s'agit de montrer l'égalité $w_{1,d} = w_{2,d}\,\Res(\uP)$ qu'il est commode d'écrire sous la forme
$$
\frac{w_{1,d}}{w_{2,d}} \overset{?}{=} \Res(\uP)
$$
D'après l'hypothèse de récurrence, on a :
$$
\forall\, d' \geqslant \delta' + 1, \qquad 
\frac{w'_{1,d'}} {w'_{2,d'}} = \Res(\uP') 
$$
Réécrivons la formule du produit pour le quotient $w_{1,d}/w_{2,d}$:
$$
\frac{w_{1,d}}{w_{2,d}} = \prod_{d'=d-(d_n-1)}^{d} \frac{w'_{1,d'}} {w'_{2,d'}}
$$
Dans ce produit, $d'$ est supérieur à $\delta'+1$ : 
en effet, $d\geqslant \delta+1$ et $\delta ' = \delta - (d_n-1)$.
De plus, le nombre de facteurs dans ce produit vaut $d_n$. 
On a donc :
$$
\frac{w_{1,d}}{w_{2,d}} = \Res(\uP')^{d_n}
$$
qui n'est autre que $\Res(\uP)$, d'après les rappels faits avant la preuve.
\end {proof}

\begin {proof}  [\bf Preuve du théorème en degré $\delta$.]
La forme $\omega$ est multiple de $\omegares$
et l'évaluation $\calQ_\delta(\uP)=\omegares(X^\emouton)$, diviseur du scalaire régulier
$w_{1,\delta}= \omega(X^\emouton)$, est un élément régulier.
Pour montrer l'égalité $\omega = w_{2,\delta}\,\omegares$, il suffit donc,
par évaluation en $X^\emouton$, de prouver que $w_{1,\delta} = w_{2,\delta}\,\calQ_\delta(\uP)$.

\noindent
\'Ecrivons l'égalité à prouver sous la forme :
$$
\frac{w_{1,\delta}}{w_{2,\delta}} \overset{?}{=} \calQ_\delta(\uP)
$$
La formule du produit pour le quotient $w_{1,d}/w_{2,d}$
avec $d = \delta$ s'écrit :
$$
\frac{w_{1,\delta}}{w_{2,\delta}} 
\ = \ 
\Bigg(\frac{w'_{1,\delta'}} {w'_{2,\delta'}}\Bigg) \times 
\Bigg(\prod_{j=\delta' + 1}^{\delta} \frac{w'_{1,j}} {w'_{2,j}}\Bigg)
$$
Or, pour $P'$, l'hypothèse de récurrence et la formule de Macaulay en degré $> \delta'$ fournissent:
$$
\frac{w'_{1,\delta'}} {w'_{2,\delta'}} = \calQ_{\delta'}(\uP')
\qquad \text{ et } \qquad
\forall\, j \geqslant \delta' + 1, \quad \frac{w'_{1,j}} {w'_{2,j}} = \Res(\uP') 
$$
Le nombre de facteurs dans le produit étant $\delta - \delta' = d_n-1$, on
obtient 
$$
\frac{w_{1,\delta}}{w_{2,\delta}} 
\ =\ 
\calQ_{\delta'}(\uP')\,  \Res(\uP')^{d_n-1}
$$
ce qui, comme par miracle, vaut $\calQ_\delta(\uP)$, cf. les rappels réalisés avant la preuve.
\end {proof}

\begin {coro} \label{DetW2diviseDetW1DegreDelta} \leavevmode
\begin {enumerate} [\rm i)]
\item  
Pour $d \geqslant \delta$, $\det W_{2,d}(\uP)$ divise $\det W_{1,d}(\uP)$.

\item
Le quotient $\det W_{1,d} / \det W_{2,d}$ est
indépendant de $d \geqslant \delta+1$ au sens de la nullité du déterminant:
$$
\forall\, d,d' \geqslant \delta + 1, \qquad 
\begin {vmatrix}
\det W_{1,d}  & \det W_{1,d'}  \\
\det W_{2,d}  & \det W_{2,d'}  \\  
\end {vmatrix} = 0
$$
\end{enumerate}
\end {coro}

\subsection{Une autre décomposition qualifiée de \og $\End$-décomposition \fg}
\label {sousSectionEndPartition}

Un des objectifs de cette section est d'obtenir, pour $h \geqslant 2$,
une formule de récurrence pour le déterminant de Macaulay $\det
W_{h,d}(\uP)$ comme produit d'autres déterminants de Macaulay $\det
W_{h-1,d'}(\uQ)$ où $d' < d$ est variable, le système $\uQ$ étant
également variable, de taille $< n$, et obtenu en spécialisant à~$0$ 
des indéterminées de la fin.

Nous verrons également dans la section suivante comment la \og
$\End$-décomposition \fg{} étudiée ici participe à l'établissement de
l'égalité des invariants de MacRae des $(\bfB_d)_{d\ge \delta+1}$.

\smallskip
Il y a des analogies, mais des analogies seulement, avec la
section~\ref{sousSectionLastPartition}: il y intervenait
$\uQ = \uP^{(h)}$ fixe et des endomorphismes $W_{h'}(\uP^{(h)})$ avec des $h' < h$
variables, tandis qu'ici on passe de $h$ à $h-1$ avec des systèmes $\uQ$
variables.

\smallskip

Rappelons que dans cette section~\ref{sousSectionLastPartition}, on
fixait $h \geqslant 2$ et regroupait les monômes $X^\alpha$ selon
les $h-1$ derniers éléments de $\alpha = (\alpha_1, \dots, \alpha_n)$.
Ici, on va procéder différemment en considérant pour un monôme
$X^\alpha \in \Jex_1$, la sous suite $\End(X^\alpha) := (\alpha_j,
\dots, \alpha_n)$, où $j = \maxDiv(X^\alpha)$.  Pour des raisons
techniques, on va imposer $X^\alpha \in \Jex_2$ de sorte que $j
\geqslant 2$. Ainsi $\bbN^\hderniers$ va être remplacé en première
approximation par la réunion disjointe $\bigsqcup_{j=2}^n
\bbN^{[j..n]} = \bbN^{[2..n]} \sqcup \cdots \sqcup \bbN^{[n-1..n]}
\sqcup \bbN^{[n..n]}$ et les monômes $X^\alpha$ de $\Jex_2$ seront
regroupés par égalité de~$\End(X^\alpha)$.

\label {NOTA13-End}%

\smallskip

Nous adoptons la lettre $\calE$ (pour end), $\calE = \End(X^\alpha)$,
versus la lettre $\calL$ (pour last), $\calL = \Last_h(X^\alpha)$.
Signalons que nous allons traiter le cas plus général des 
déterminants $\det W_\calM$ en imposant à $\calM$ d'être un
sous-module monomial de $\Jex_2$  et que nous particulariserons ensuite
$\calM$ à $\Jex_{h,d}$ avec $h \geqslant 2$.

\smallskip

Voici des définitions précises qui corrigent en particulier le lieu
de vie des $\End(X^\alpha) = (\alpha_j, \dots, \alpha_n)$, où $j =
\maxDiv(X^\alpha)$, en ayant remarqué que $\alpha_j \ge d_j$ et
$\alpha_k < d_k$ pour $k > j$.

\begin{defns}\label{EndOrder}
On note $\scrE=\scrE(D)$ le sous-ensemble de $\bigsqcup\limits_{j=2}^n \bbN^{[j..n]}$ constitué des $\calE =
(\alpha_j, \dots, \alpha_n)$ vérifiant:
$$
\alpha_j \geqslant d_j, \qquad  \text{ et } \quad 
\quad \forall\  k > j, \quad \alpha_k < d_k  
$$  
Si $\calE \in \bbN^{[j..n]}$, $\calE' \in \bbN^{[j'..n]}$, alors
$\#\calE = n-j+1$, $\#\calE' = n-j'+1$ de sorte que l'on
a l'équivalence $\#\calE > \#\calE' \iff j < j'$.

On munit $\scrE$ de la relation d'ordre notée $\preccurlyeq$:
$$
\calE \preccurlyeq \calE' : \quad
\left\{
\begin{array}{l}
\#\calE > \#\calE' \\ 
\text{ou} \\ 
\#\calE = \#\calE' \text{ et \ } \calE_k \leqslant \calE'_k \text{ pour tout $k$} \\
\end{array}
\right.
$$
Pour un sous-module monomial $\calM \subset \Jex_2$, on désigne par
$\End(\calM)$ la partie finie de $\scrE$ définie par $\End(\calM) =
\{\End(X^\alpha) \mid X^\alpha \in \calM\}$.

Pour $\calE \in \scrE$, on note $\calME$ le sous-module
monomial de $\calM$ de base les $X^\alpha$ tels que $\End(X^\alpha) =
\calE$, de sorte que l'on a la décomposition qualifiée de
{\rm $\End$-décomposition} :
$$
\calM = \bigoplus_{\calE \in \scrE} \calME
$$
Note: on a $\calME = 0$ pour $\calE \notin \End(\calM)$.
\end{defns}

\label {NOTA13-EndM}%
\label {NOTA13-McalE}%

\medskip

La End-décomposition est un raffinement de la $\maxDiv$-décomposition
dans le sens où $\calME \subset \calM\double{j}$ en ayant posé $j =
n-\#\calE + 1$.  La End-décomposition est également un raffinement de
la Last-décomposition dans le sens où $\calME \subset
\calM\crochet\calL$ où $\calL$ est n'importe quel suffixe de $\calE$.


\subsubsection*{Exemple: $\calM = \Jex_{2,d}(D)$ où $D=(4,2,3,3)$}

Dans le tableau ci-dessous, nous regroupons les
$\Jex_{2,d}\double\calE$ (non nuls) de même dimension.  Par exemple,
pour $d=7$, pour lequel $\dim \Jex_{2,d} = 30$, la suite $(1^8, 2^2,
4^2, 10)$ est une suite de dimensions indiquant qu'il y a 8 sous-espaces $\Jex_{2,d}\double\calE$ de
dimension 1, 2 de dimension 2, 2 de dimension 4 et 1 de dimension
10. Et $8\times1 + 2\times2 + 2\times4 + 1\times10 =30$.
$$
\begin {array}{c|c|c|c|c|c}
  d          &5 &6 &7 &8 &9
\\[1mm]
\hline
\vrule height10pt depth3pt width0pt
\dim \Jex_{2,d} &2 &10 &30 &66 &118
\\[1mm]
\hline
\vrule height10pt depth3pt width0pt
\text{dims } \Jex_{2,d}\double\calE& 1^2 & 1^4,2,4 &1^8,2^2,4^2,10 &1^{14},2^3,4^3,6,10,18
            &1^{22},2^3,4^4,6^2,7,10,18,27
\end {array}
$$
Nous allons voir (points ii) et iii) de \ref{EndProperties}) que le
déterminant $\det W_{2,7}$  i.e. $\det W_\calM$ avec
$\calM=\Jex_{2,7}$, de dimension 30, est le produit des déterminants $\det W_\calME$, de
dimensions $1,2,4,10$. C'est encore plus spectaculaire avec $d=9$
où le déterminant de dimension 118 est le produit de déterminants
de dimensions bien plus petites.

\medskip

Pour $d=6$, la suite des dimensions $(1^4,2,4)$ précise que les 10
monômes $X^\alpha$ de $\Jex_{2,6}$ sont partitionnés en $4+1+1$ paquets de même
$\End(X^\alpha)$, 4 de cardinal~1, 1 de cardinal~2 et 1 de cardinal~4.
L'utilisation du symbole $\dot X_j$ indique que $j = \maxDiv(X^\alpha)$:
$$
\def\udX{\dot X}
X_2^2\udX_4^4,
\quad
X_2^2\udX_3^4,
\quad
X_2^2\udX_3^3X_4,
\quad
X_1^4\udX_2^2,
\quad
X_1X_2^2\udX_3^3,\ X_2^3\udX_3^3,
\quad
X_1X_2^2\udX_4^3,\ X_2^3\udX_4^3,\ X_2^2X_3\udX_4^3,\ X_3^3\udX_4^3
$$
Si on permute $D$, cela modifie $\Jex_{2,d}(D)$ sans changer sa dimension,
et du coup la combinatoire des $\calME$ peut être
radicalement transformée. Par exemple, pour $D = (2,3,3,4)$ au lieu de $(4,2,3,3)$,
on a:
$$
\begin {array}{c|c|c|c|c|c}
  d          &5 &6 &7 &8 &9
\\[1mm]
\hline
\vrule height10pt depth3pt width0pt
\text{dims } \Jex_{2,d}\double\calE& 1^2 & 1^7,3 &1^{14},3^2,5^2 &1^{24},3^3,5^3,6,12
            &1^{35},3^4,5^4,6^2,7,12,20
\end {array}
$$

\begin{prop} \label{EndTriangularite}
L'endomorphisme $W_\calM=W_\calM(\uP)$ est triangulaire relativement à 
$\calM = \bigoplus_{\calE \in \scrE} \calME$
au sens où, pour $\calE \in \scrE$ fixé, on a 
$$
W_{\calM}\big(\calME \big) 
\ \subset \ 
\bigoplus_{\calE' \succcurlyeq \calE}\, \calM\double{\calE'}
$$ 
\end{prop}

\begin{proof}
  
Soit $X^\alpha \in \calM$, $i = \minDiv(X^\alpha)$ et $j =
\maxDiv(X^\alpha)$ de sorte que $i < j$.  L'image $W_\calM(X^\alpha)$
est une combinaison linéaire de monômes $X^{\alpha'}\in \calM$ du type
$(X^\alpha/X_i^{d_i}) X^\gamma$ avec $X^\gamma$ monôme de $P_i$.

\noindent
Il faut montrer que
$\End(X^{\alpha}) \preccurlyeq \End(X^{\alpha'})$.  En notant $j' =
\maxDiv(X^{\alpha'})$, puisque $i < j$, on a $j \leqslant j'$.  Si
$j<j'$, c'est terminé par définition de $\preccurlyeq$.  Sinon $j =
j'$ et:
$$
\End(X^\alpha) = (\alpha_j, \alpha_{j+1}, \dots, \alpha_n)
\qquad \qquad
\End(X^{\alpha'}) = 
(\alpha_j + \gamma_j, \alpha_{j+1} + \gamma_{j+1}, \dots, \alpha_n + \gamma_n)
$$
donc $\End(X^{\alpha}) \preccurlyeq \End(X^{\alpha'})$.
\end{proof}

\medskip

Pour $\calE \in \scrE$, on pose $\bfA[\uX_\calE] = \bfA[X_1,\dots,X_{n-\#\calE}]$.
Vont à présent cohabiter le $n$-système $\uP = (P_1, \dots, P_n)$  de $\bfA[\uX]$
et le système $\uP_\calE$ de $\bfA[\uX_\calE]$ défini par:
$$
\uP_\calE = (P_{1,\calE}, \dots, P_{n-\#\calE,\,\calE})
\qquad \text{où} \qquad
P_{i,\calE} = P_i(X_1, \dots, X_{n-\#\calE}, 0, \dots, 0)
$$

\begin{prop} \label{EndProperties}
\leavevmode

\begin{enumerate}[\rm i)]
\item 
Pour tout $\calE \in \scrE$, les endomorphismes 
$W_\calME(\uP)$ et 
$W_{\calME/X^\calE}(\uP_\calE)$ sont 
conjugués et ont donc même déterminant.

\medskip
\item 
On a l'égalité déterminantale 
$$
\det W_\calM(\uP) \ = \ 
\prod_{\calE \in \scrE} 
\det W_{\calME/X^\calE}(\uP_\calE)
$$

\item 
Pour $\calM = \Jex_{h,d}$, on a 
$$
\det W_{h,d}(\uP) \ = \  \prod_{\calE\in\scrE} \det W_{h-1,\, d-|\calE|}(\uP_\calE)
\ = \  \prod_{\calE\in\End(\Jex_{h,d})} \kern -5pt \det W_{h-1,\, d-|\calE|}(\uP_\calE)
$$
\end{enumerate}
\end{prop}

\begin{proof}
i) 
L'application  $\Psi$ suivante conserve $\minDiv$ :
$$
\Psi : \ 
\begin{array}[t]{rcl}
\calM & \longrightarrow & \bfA[\uX]\\ [0.2cm]
X^{\alpha} & \longmapsto & \dfrac{X^{\alpha}}{X^{\calE_\alpha}}
\ \in \ \bfA[\uX_{\calE_\alpha}] \quad \text{où $\calE_\alpha = \End(X^\alpha)$}
\end{array}
$$
Pour un $\calE \in \scrE$ fixé,
sa restriction $\Psi_\calE$ à chaque $\calME$ est injective et
induit une bijection dont la réciproque est la multiplication par
$X^\calE$:
$$
\Psi_\calE : \ 
\begin{array}[t]{rcl}
\calME & \longrightarrow & \calME/X^\calE \subset \bfA[\uX_\calE]\\ [0.2cm]
X^{\alpha} & \longmapsto & \dfrac{X^{\alpha}}{X^{\calE}}
\end{array}
$$
Soit $X^\alpha \in \calME$.  Comme $X^\alpha \in \calM \subset
\Jex_2$, l'indice $i := \minDiv(X^\alpha)$ est strictement inférieur à
$j := \maxDiv(X^\alpha)$, de sorte que
$\End\big(X^\alpha/X_i^{d_i}\big) = \calE$.

\smallskip

Il s'agit d'assurer au niveau de la flèche pointillée la commutativité du diagramme suivant:
$$
\xymatrix @R = 1.2cm @C = 3cm @M=0.4pc{
X^{\alpha} \ar[d]_{W_\calME(\uP)} \ar[r]^{\Psi_{\calE}} 
& \dfrac{X^{\alpha}}{X^\calE}  \ar[d]^{W_{\calME/X^\calE}(\uP_\calE)}
\\
\pi_\calME \Big(\dfrac{X^{\alpha}}{X_i^{d_i}} \,P_i\Big)
\ar@{-->}[r]^-{\Psi_{\calE}} & 
\pi_{\calME/X^\calE} \Big( \dfrac{X^{\alpha}}{X_i^{d_i}X^{\calE}}\, P_{i,\calE} \Big)
\\
}
$$
L'élément en bas à gauche est une combinaison linéaire de monômes
$X^{\alpha'} \in \calME$ du type $X^{\alpha'} =
(X^\alpha/X_i^{d_i})\,X^\gamma$ où $X^\gamma$ est un monôme de~$P_i$.
Comme $\End\big(X^\alpha/X_i^{d_i}\big) = \calE$ et $X^{\alpha'} \in
\calME$, on en déduit que le monôme $X^\gamma$ ne dépend pas des
$\#\calE$ dernières variables, donc est un monôme de~$P_{i,\calE}$.
Ainsi, $\dfrac{X^{\alpha'}}{X^\calE}$ est dans $\pi_{\calME/X^\calE}
\Big( \dfrac{X^{\alpha}}{X_i^{d_i}X^{\calE}}\, P_{i,\calE} \Big)$.

\medskip
ii) 
La décomposition triangulaire utilisée dans la proposition précédente~\ref{EndTriangularite} fournit:
$$
\det W_\calM(\uP) \ = \ 
\prod_{\calE \in \scrE} \det W_\calME(\uP)
$$
et d'après le point i), on a $\det W_\calME(\uP) = \det
W_{\calME/X^\calE}(\uP_\calE)$ d'où le résultat.

\medskip
iii) 
Soit $\calM=\Jex_{h,d}$. Pour $\calE = (\alpha_j, \dots, \alpha_n) \in \scrE$,
montrons, dans $\bfA[\uX_\calE]$, que l'on a l'égalité
$\calME/X^\calE =  \Jex_{h-1,\,d-|\calE|}$.
Il en résultera $W_{\calME/X^\calE}(\uP_\calE) =  W_{h-1, d-|\calE|}(\uP_\calE)$.

\smallskip

Rappelons, par définition de $\scrE$, que l'on a $\alpha_j \geqslant
d_j$ et $\alpha_k < d_k$ pour $k > j$.  Un monôme $X^\alpha \in
\calME$ est de la forme $X^\alpha = X^\beta X^\calE$, vérifie $j =
\maxDiv(X^\alpha)$ et $\DivSeq(X^\alpha) \setminus \{j\} \subset
\DivSeq(X^\beta)$.
En conséquence $X^\beta \in \Jex_{h-1,\,d-|\calE|}$, d'où
l'inclusion $\calME/X^\calE \subset \Jex_{h-1,\,d-|\calE|}$. Pour l'inclusion
opposée, soit $X^\beta$ un monôme de $\bfA[\uX_\calE]$ appartenant à $\Jex_{h-1,\,d-|\calE|}$;
en posant $X^\alpha = X^\beta X^\calE$, on a, en raison du rappel, d'une part
$X^\alpha \in \Jex_{h,d}$ et d'autre part $\End(X^\alpha) = \calE$.
\end{proof}

\medskip
Le théorème qui suit a déjà été énoncé et démontré en~\ref{NiveauhDivisehmoins1}.
On peut remarquer que la récurrence ci-dessous est d'une autre nature
mais l'initialisation $h=2$, non triviale, reportée
en~\ref{DetW2diviseDetW1}, est la même.

\begin{theo}
Pour tout $d$ et tout $h \geqslant 2$, on a
$$
\det W_{h,d}(\uP) \ \mid \ \det W_{h-1,d}(\uP)
$$
\end{theo}

\begin{proof}
Procédons par récurrence sur $h$, avec l'hypothèse portant sur tout $d$ et tout système $\uP$.
Pour l'hérédité, on utilise~\ref{EndProperties}-iii) qui fournit :
$$
\det W_{h+1,d}(\uP) \ = \ 
\prod_{\calE\in\scrE} \det W_{h,d-|\calE|}(\uP_\calE),
\qquad\quad 
\det W_{h,d}(\uP) \ = \ 
\prod_{\calE\in\scrE} \det W_{h-1,d-|\calE|}(\uP_\calE)
$$
D'après l'hypothèse de récurrence appliquée au système $\uP_\calE$, on obtient:
$$
\det W_{h, d-|\calE|}(\uP_\calE)
\ \mid \ 
\det W_{h-1, d-|\calE|}(\uP_\calE)
$$
D'où $\det W_{h+1,d}(\uP) \mid \det W_{h,d}(\uP)$.
\end{proof}

\subsubsection{Ordres totaux, extensions linéaires de la relation d'ordre partiel
$(\scrE, \preccurlyeq)$ définie en \ref{EndOrder}}

\begin{lem}
Considérons sur chaque $\bbN^{[j..n]}$ une structure d'ordre monomiale que nous
notons~$\leqslant_{m}$ (indépendamment de $j$).
Et définissons une relation $\leqslant$ sur $\scrE$ par:
$$
\calE \leqslant \calE' : \quad
\left\{
\begin{array}{l}
j < j' \\ 
\text{ou} \\ 
j = j' \text{ et \ } \calE \leqslant_m \calE' \\
\end{array}
\right.
\qquad \text{où}\quad  j = n -\#\calE + 1,\quad j' = n -\#\calE' + 1
$$
Alors $\leqslant$ est une relation d'ordre total qui vérifie
$\calE \preccurlyeq \calE' \Rightarrow \calE \leqslant \calE'$.
\end{lem}

\begin{proof}
Supposons $\calE \preccurlyeq \calE'$ donc $j \leqslant j'$. Si $j < j'$,
alors $\calE \leqslant \calE'$. Si $j = j'$,
on peut écrire $\calE' = \calE + \calE''$ pour un certain $\calE'' \in \mathbb N^{[j..n]}$. 
Comme $\leqslant_m$ est un ordre monomial, on a 
nécessairement $\calE \leqslant_m \calE + \calE''$, c'est-à-dire $\calE \leqslant_m \calE'$.
\end{proof}

\begin{rmq}\label{RangementEnd}
Par exemple, en prenant l'ordre monomial \textsc{Lex}, on obtient l'écriture suivante :
$$
\calE \leqslant \calE'  
\ \Leftrightarrow \ 
(j,\calE) \, \leqslant_{\rm lex} (j',\calE')
\qquad \hbox{avec $j = n -\#\calE + 1$ et $j' = n -\#\calE' + 1$}
$$
On peut également considérer l'ordre total suivant, de type \textsc{GrevLex}.
En notant $\Rev\calE$ la suite renversée de~$\calE$, il est défini par 
$$
\calE \leqslant \calE' : \quad
\left\{
\begin{array}{l}
j < j' \\ 
\text{ou} \\ 
j = j' \text{ et \ } \Rev{\calE} \geqslant_{\rm lex} \Rev{\calE'}
\end{array}
\right.
$$
En page~\pageref{RangementEndSuivantGrevLex}, on utilisera que ce dernier ordre total $\leqslant$
possède la propriété:
$$
\left\{
\begin{array}{l}
| \alpha | = | \alpha' |   \\ [0.1cm]
X^\alpha \ <_{\textsc{\tiny GrevLex}} \ X^{\alpha'} \\[0.1cm]
\maxDiv(X^\alpha) = \maxDiv(X^{\alpha'}) \\
\end{array}
\right\}
\quad \Longrightarrow \quad
\End(X^\alpha) \leqslant \End(X^{\alpha'})
$$
Cette implication dit grosso-modo qu'à $\maxDiv$-égal, l'ordre
monomial \textsc{GrevLex} a un impact sur la relation d'ordre total
$\leqslant$ de $\scrE$.
En revanche, \textsc{GrevLex} est sans influence
directe sur $\maxDiv$ (ce qui est normal car l'un dépend de $D$ et
l'autre pas).  Par exemple, pour $D=(1,1,2,2)$, on a
$$
X_1 X_2 X_3^2 \  >_{\textsc{\tiny GrevLex}}  \ X_1 X_2 X_3 X_4 \  
>_{\textsc{\tiny GrevLex}} \ X_2 X_3 X_4^2
$$
mais les $\maxDiv$ respectifs sont $3$, $2$ et $4$.
Plus généralement, en considérant $\alpha = (\dots, d_{n-1}, d_n - 2)$, 
$\alpha' = (\dots, d_{n-1}-1, d_n-1)$ et $\alpha'' = (\dots, d_n)$, 
on a 
$$
X^\alpha \ >_{\textsc{\tiny GrevLex}} \ X^{\alpha'}  \ >_{\textsc{\tiny GrevLex}} \  X^{\alpha''} 
\quad \hbox{ et } \quad
\maxDiv(X^{\alpha'}) \, < \,
\maxDiv(X^{\alpha})=n-1 \, < \,
\maxDiv(X^{\alpha''})=n
$$
\end{rmq}

\subsection{L'égalité de stabilité de l'invariant de MacRae de $\bfB_d$ en degré $d\geqslant \delta+1$}
\label{sectionEgaliteStabiliteMacRae}

\index{théorème!e2@égalité de stabilité de l'invariant de MacRae des $\bfB_d$ en degré $d\ge\delta+1$}%

On fixe un système $\uP$ n'apparaissant pas dans les notations. On
rappelle (cf. les deux premières sections du
chapitre~\ref{ChapMacRaeForP}) que pour $d\ge\delta+1$, on a attaché à
$\uP$ un scalaire $\calR_d$ comme spécialisation du cas générique, cas
dans lequel $\calR_d$ est l'invariant de MacRae normalisé de~$\bfB_d$,
$\bfA$-module de MacRae de rang~0 (cf. le
théorème~\ref{PoidsNormalisationMacRae}).  Et nous avons montré
en~\ref{MacRaeEqualities} la propriété suivante d'invariance par
rapport à $d$:
$$
\forall\, d,d' \geqslant \delta+1, \qquad   \calR_d = \calR_{d'}
$$
Ce scalaire commun est une définition \emph{possible} de $\Res(\uP)$,
cf. la définition~\ref{DefResultant} du résultant de $\uP$.

\smallskip

Par ailleurs, nous avons prouvé dans le théorème~\ref{MacaulayParRecurrence}:
$$
\forall\,d \geqslant \delta+1, \qquad
\calR_d = \dfrac{\det W_{1,d}}{\det W_{2,d}} 
$$
Nous avons donc l'égalité des quotients ci-dessous, égalité qui doit être comprise sous la forme
de droite (exprimée dans le corollaire~\ref {DetW2diviseDetW1DegreDelta}):
$$
\forall\, d,d' \geqslant \delta+1, \qquad
\dfrac{\det W_{1,d}}{\det W_{2,d}} =  \dfrac{\det W_{1,d'}}{\det W_{2,d'}}
\qquad\qquad
\begin {vmatrix}
\det W_{1,d}  & \det W_{1,d'}  \\
\det W_{2,d}  & \det W_{2,d'}  \\  
\end {vmatrix} = 0
$$
On peut supposer $\uP$ générique, ce qui a l'avantage de rendre les déterminants
réguliers, auquel cas la nullité des déterminants de la famille ci-dessus est équivalente
à celle  pour $d' = d+1$:
$$
\forall\, d \geqslant \delta+1, \qquad
\begin {vmatrix}
\det W_{1,d}  & \det W_{1,d+1}  \\
\det W_{2,d}  & \det W_{2,d+1}  \\  
\end {vmatrix} = 0
$$
L'objectif de cette section est de donner pour cette dernière égalité une
preuve \emph {directe}, totalement autonome et indépendante des
notions de résultant et d'invariant de MacRae et qui repose uniquement
sur les matrices excédentaires (variantes des matrices de
Macaulay).

\medskip

Nous avons essentiellement besoin de rappeler quelques
définitions: $\Jex_1^\circ$ est le sous-module monomial de base les
monômes $X^\beta \in \Jex_1$ tels que $\beta_n = 0$
(section~\ref{SousSectionDecompositionJ1circJ1plus}) et
$\Jex_2^\eq$ celui de base les monômes $X^\alpha \in \Jex_2$ tels que
$\alpha_j = d_j$ où $j = \maxDiv(X^\alpha)$
(section~\ref{SousSectionMaxDivDecomposition}).  Et on a montré
en~\ref{DetW1dDiviseDetW1dplus1} et \ref{DivisibiliteEntreW2} que
$$
\begin{array}{ccccl}
\det W_{1,d+1} & = & \det W_{1,d} \times \det W_{\calM} & &
\hbox{avec $\calM = \Jex_{1,d+1}^\circ$ et $d \geqslant \delta+1$} \\ [0.4cm]
\det W_{2,d+1} & = & \det W_{2,d} \times \det W_{\calN} & & 
\hbox{avec $\calN = \Jex_{2,d+1}^\eq$ et $d \in \bbN$} 
\end{array}
$$
L'invariance par rapport à $d$ dont nous entendons donner une preuve
directe est donc équivalente à l'identité algébrique $\det W_\calM
= \det W_\calN$.  Avec les notations précédentes et en faisant le
changement $d \leftrightarrow d+1$, cette identité s'écrit :
$$
\det W_{2,d}^\eq \ =\ \det W_{1,d}^\circ
\qquad \text{pour tout $d \geqslant \delta + 2$}
$$

\label {NOTA13-Jex1o}%
\label {NOTA13-Jex2=}%
\label {NOTA13-W1o}%
\label {NOTA13-W2=}%

\noindent
C'est cette égalité de déterminants que nous allons étudier ici 
mais surtout pas en s'attaquant directement aux déterminants !
Plutôt en étudiant les \textit{endomorphismes} 
$W_{2,d}^\eq$ et $W_{1,d}^\circ$, qui possèdent une jolie propriété 
que nous illustrons avec $D=(1,1,2,2)$ et $d = \delta+2 = 4$.
Après un petit travail (sic), voici ce que l'on obtient :

{\small
$$
\renewcommand \Heti[1] {\omit \quad \mbox{\scriptsize$#1$} \hfil}  
W_{2,d}^\eq \ = \ 
\EastBordermatrix{
a_{1} & \VR . & \VR . & \VR . & \VR . & . & . & \VR . & . & \VR . & . & . & . & . & . & \Heti{X_{1}^{3}X_{2}} \\ 
\HR{20} 
a_{4} & \VR a_{1} & \VR . & \VR . & \VR . & . & . & \VR . & . & \VR . & . & . & . & . & . & \Heti{X_{1}^{2}X_{2}X_{4}} \\ 
\HR{20} 
a_{3} & \VR . & \VR a_{1} & \VR . & \VR . & . & . & \VR . & . & \VR . & . & . & . & . & . & \Heti{X_{1}^{2}X_{2}X_{3}} \\ 
\HR{20} 
. & \VR a_{3} & \VR a_{4} & \VR a_{1} & \VR . & . & . & \VR . & . & \VR . & . & . & . & . & . & \Heti{X_{1}X_{2}X_{3}X_{4}} \\ 
\HR{21} 
. & \VR . & \VR . & \VR . & \VR a_{1} & . & . & \VR . & . & \VR . & . & . & . & . & . & \Heti{X_{1}^{2}X_{3}^{2}} \\ 
. & \VR . & \VR a_{3} & \VR . & \VR a_{2} & a_{1} & b_{1} & \VR . & . & \VR . & . & . & . & . & . & \Heti{X_{1}X_{2}X_{3}^{2}} \\ 
. & \VR . & \VR . & \VR . & \VR . & a_{2} & b_{2} & \VR . & . & \VR . & . & . & . & . & . & \Heti{X_{2}^{2}X_{3}^{2}} \\ 
\HR{20} 
. & \VR . & \VR . & \VR . & \VR a_{4} & . & . & \VR a_{1} & b_{1} & \VR . & . & . & . & . & . & \Heti{X_{1}X_{3}^{2}X_{4}} \\ 
. & \VR . & \VR . & \VR a_{3} & \VR . & a_{4} & b_{4} & \VR a_{2} & b_{2} & \VR . & . & . & . & . & . & \Heti{X_{2}X_{3}^{2}X_{4}} \\ 
\HR{21} 
. & \VR . & \VR . & \VR . & \VR . & . & . & \VR . & . & \VR a_{1} & . & . & . & . & c_{1} & \Heti{X_{1}^{2}X_{4}^{2}} \\ 
. & \VR a_{4} & \VR . & \VR . & \VR . & . & . & \VR . & . & \VR a_{2} & a_{1} & . & b_{1} & . & c_{2} & \Heti{X_{1}X_{2}X_{4}^{2}} \\ 
. & \VR . & \VR . & \VR . & \VR . & . & . & \VR . & . & \VR a_{3} & . & a_{1} & . & b_{1} & c_{3} & \Heti{X_{1}X_{3}X_{4}^{2}} \\ 
. & \VR . & \VR . & \VR . & \VR . & . & . & \VR . & . & \VR . & a_{2} & . & b_{2} & . & c_{5} & \Heti{X_{2}^{2}X_{4}^{2}} \\ 
. & \VR . & \VR . & \VR a_{4} & \VR . & . & . & \VR . & . & \VR . & a_{3} & a_{2} & b_{3} & b_{2} & c_{6} & \Heti{X_{2}X_{3}X_{4}^{2}} \\ 
. & \VR . & \VR . & \VR . & \VR . & . & . & \VR a_{4} & b_{4} & \VR . & . & a_{3} & . & b_{3} & c_{8} & \Heti{X_{3}^{2}X_{4}^{2}} \\ 
\noalign{\vskip-1pt}
}
$$
$$
W_{1,d}^\circ \ = \ 
\EastBordermatrix{
a_{1} & \VR . & \VR . & \VR . & \VR . & . & . & \VR . & . & \VR . & . & . & . & . & . & \Heti{X_{1}^{4}} \\ 
\HR{20} 
a_{3} & \VR a_{1} & \VR . & \VR . & \VR . & . & . & \VR . & . & \VR . & . & . & . & . & . & \Heti{X_{1}^{3}X_{3}} \\ 
\HR{20} 
a_{2} & \VR . & \VR a_{1} & \VR . & \VR . & . & . & \VR . & . & \VR . & . & . & . & . & . & \Heti{X_{1}^{3}X_{2}} \\ 
\HR{20} 
. & \VR a_{2} & \VR a_{3} & \VR a_{1} & \VR . & . & . & \VR . & . & \VR . & . & . & . & . & . & \Heti{X_{1}^{2}X_{2}X_{3}} \\ 
\HR{20} 
. & \VR . & \VR a_{2} & \VR . & \VR a_{1} & . & . & \VR . & . & \VR . & . & . & . & . & . & \Heti{X_{1}^{2}X_{2}^{2}} \\ 
. & \VR . & \VR . & \VR . & \VR a_{2} & a_{1} & b_{1} & \VR . & . & \VR . & . & . & . & . & . & \Heti{X_{1}X_{2}^{3}} \\ 
. & \VR . & \VR . & \VR . & \VR . & a_{2} & b_{2} & \VR . & . & \VR . & . & . & . & . & . & \Heti{X_{2}^{4}} \\ 
\HR{20} 
. & \VR . & \VR . & \VR a_{2} & \VR a_{3} & . & . & \VR a_{1} & b_{1} & \VR . & . & . & . & . & . & \Heti{X_{1}X_{2}^{2}X_{3}} \\ 
. & \VR . & \VR . & \VR . & \VR . & a_{3} & b_{3} & \VR a_{2} & b_{2} & \VR . & . & . & . & . & . & \Heti{X_{2}^{3}X_{3}} \\ 
\HR{20} 
. & \VR a_{3} & \VR . & \VR . & \VR . & . & . & \VR . & . & \VR a_{1} & . & . & . & . & c_{1} & \Heti{X_{1}^{2}X_{3}^{2}} \\ 
. & \VR . & \VR . & \VR a_{3} & \VR . & . & . & \VR . & . & \VR a_{2} & a_{1} & . & b_{1} & . & c_{2} & \Heti{X_{1}X_{2}X_{3}^{2}} \\ 
. & \VR . & \VR . & \VR . & \VR . & . & . & \VR . & . & \VR a_{3} & . & a_{1} & . & b_{1} & c_{3} & \Heti{X_{1}X_{3}^{3}} \\ 
. & \VR . & \VR . & \VR . & \VR . & . & . & \VR a_{3} & b_{3} & \VR . & a_{2} & . & b_{2} & . & c_{5} & \Heti{X_{2}^{2}X_{3}^{2}} \\ 
. & \VR . & \VR . & \VR . & \VR . & . & . & \VR . & . & \VR . & a_{3} & a_{2} & b_{3} & b_{2} & c_{6} & \Heti{X_{2}X_{3}^{3}} \\ 
. & \VR . & \VR . & \VR . & \VR . & . & . & \VR . & . & \VR . & . & a_{3} & . & b_{3} & c_{8} & \Heti{X_{3}^{4}} \\ 
\noalign{\vskip-1pt}
}
$$
}

\noindent
Tout ceci est relatif à l'écriture suivante des polynômes ($P_4$ n'intervient pas):

\smallskip

{%
\setlength{\tabcolsep}{2pt}
\noindent
$
\left\{
\begin{tabular}{rcp{15cm}} 
$P_{1}$ & $=$ & $a_{1}X_{1} + a_{2}X_{2} + a_{3}X_{3} + a_{4}X_{4}$\\ [0.1cm] 
$P_{2}$ & $=$ & $b_{1}X_{1} + b_{2}X_{2} + b_{3}X_{3} + b_{4}X_{4}$\\ [0.1cm] 
$P_{3}$ & $=$ & $c_{1}X_{1}^{2} + c_{2}X_{1}X_{2} + c_{3}X_{1}X_{3} + c_{4}X_{1}X_{4} + 
c_{5}X_{2}^{2} + c_{6}X_{2}X_{3} + c_{7}X_{2}X_{4} + c_{8}X_{3}^{2} + 
c_{9}X_{3}X_{4} + c_{10}X_{4}^{2}$\\ [0.1cm] 
\end{tabular} 
\right.
$
}

Le lecteur remarquera que, comme par magie, les matrices sont triangulaires par blocs et que 
les blocs diagonaux sont les mêmes pour les deux matrices (d'où l'égalité des déterminants).
Du côté de $\Jex_{2,d}^\eq$, les monômes sont regroupés par égalité de $\End$.
Quant à la partition des monômes de $\Jex_{1,d}^\circ$, 
c'est l'image de celle des monômes de $\Jex_{2,d}^\eq$ 
par un isomorphisme monomial construit de manière astucieuse qui s'est révélé délicat à découvrir.

\medskip

Commentons encore un peu l'exemple précédent.
Les blocs de $\Jex_{2,d}^\eq$ correspondent aux $\End$-uplets suivants :
$$
(1,0,0) \qquad 
(1,0,1) \qquad 
(1,1,0) \qquad 
(1,1,1) \qquad 
(2,0) \qquad 
(2,1) \qquad 
(2) 
$$
Ces $\End$-uplets sont rangés conformément à l'ordre total précisé
en~\ref{RangementEnd} gouverné par l'ordre lexicographique ; cet ordre
total est une extension linéaire de l'ordre partiel $\preccurlyeq$ qui
intervient dans la structure triangulaire de l'endomorphisme $W_\calM$
(cf.~\ref{EndTriangularite}).  Remarquons que l'on peut permuter les
uplets $(1,0,1)$ et $(1,1,0)$ car ils sont incomparables pour
$\preccurlyeq$.

\bigskip

\label{RangementEndSuivantGrevLex}
Donnons un autre rangement possible pour les monômes de
$\Jex_{2,d}^\eq$ toujours avec $D=(1,1,2,2)$ et $d = \delta+2 = 4$.
Conformément à~\ref{RangementEnd}, on peut utiliser l'ordre monomial
\textsc{GrevLex} en ayant pris soin de ranger les monômes suivant leur
$\maxDiv$.  Au sein d'un bloc de même $\maxDiv$ (il y en a 3 dans
l'exemple ci-dessous), on ordonne les monômes selon l'ordre monomial
\textsc{GrevLex} (ici de manière décroissante), et le découpage
suivant les $\End$-uplets se fait tout seul !  Cela se manifeste de la
façon suivante ($>$ désigne~$>_{\textsc{\tiny GrevLex}}$)
$$
\underbrace{
X_1^3X_2 
>
X_1^2X_2X_3
>
X_1^2X_2X_4
>
X_1X_2X_3X_4}_{\maxDiv = 2}
\qquad 
\underbrace{
\begin{array}{c}
X_1^2X_3^2\\
\vee \\
X_1X_2X_3^2\\
\vee \\
X_2^2X_3^2\\
\end{array}
\ > \ 
\begin{array}{c}
X_1X_3^2X_4\\
\vee \\
X_2X_3^2X_4\\
\end{array}
}_{\maxDiv = 3}
\quad 
\underbrace{
\begin{array}{c}
X_1^2X_4^2\\
\vee \\
X_1X_2X_4^2\\
\vee \\
X_2^2X_4^2\\
\vee \\
X_1X_3X_4^2\\
\vee \\
X_2X_3X_4^2\\
\vee \\
X_3^2X_4^2\\
\end{array}}_{\maxDiv = 4}
$$
Voici le rangement des $\End$-uplets (au dessus de chaque $\End$-uplet, 
on a indiqué le \og $\End$-monôme \fg{} correspondant) :
$$
\underbrace{
\overset{X_2\phantom{X_3X_4}}{(1,0,0)} \quad 
\overset{X_2X_3\phantom{X_4}}{(1,1,0)} \quad 
\overset{X_2\ X_4}{(1,0,1)} \quad 
\overset{X_2X_3X_4}{(1,1,1)}}_{\maxDiv = 2} \qquad 
\underbrace{
\overset{X_3^2\phantom{X_3}}{(2,0)} \quad 
\overset{X_3^2X_4}{(2,1)}
}_{\maxDiv = 3} \qquad
\underbrace{
\overset{X^2_4}{(2)}
}_{\maxDiv = 4}
$$
Quant aux endomorphismes $W^\eq_{2,d}(\uP)$ et $W^\sh_{1,d}(\uP)$:
{\small 
$$
W_{2,d}^\eq \ = \ 
\EastBordermatrix{
a_{1} & \VR . & \VR . & \VR . & \VR . & . & . & \VR . & . & \VR . & . & . & . & . & . & \Heti{X_{1}^{3}X_{2}} \\ 
\HR{20} 
a_{3} & \VR a_{1} & \VR . & \VR . & \VR . & . & . & \VR . & . & \VR . & . & . & . & . & . & \Heti{X_{1}^{2}X_{2}X_{3}} \\ 
\HR{20} 
a_{4} & \VR . & \VR a_{1} & \VR . & \VR . & . & . & \VR . & . & \VR . & . & . & . & . & . & \Heti{X_{1}^{2}X_{2}X_{4}} \\ 
\HR{20} 
. & \VR a_{4} & \VR a_{3} & \VR a_{1} & \VR . & . & . & \VR . & . & \VR . & . & . & . & . & . & \Heti{X_{1}X_{2}X_{3}X_{4}} \\ 
\HR{21} 
. & \VR . & \VR . & \VR . & \VR a_{1} & . & . & \VR . & . & \VR . & . & . & . & . & . & \Heti{X_{1}^{2}X_{3}^{2}} \\ 
. & \VR a_{3} & \VR . & \VR . & \VR a_{2} & a_{1} & b_{1} & \VR . & . & \VR . & . & . & . & . & . & \Heti{X_{1}X_{2}X_{3}^{2}} \\ 
. & \VR . & \VR . & \VR . & \VR . & a_{2} & b_{2} & \VR . & . & \VR . & . & . & . & . & . & \Heti{X_{2}^{2}X_{3}^{2}} \\ 
\HR{20} 
. & \VR . & \VR . & \VR . & \VR a_{4} & . & . & \VR a_{1} & b_{1} & \VR . & . & . & . & . & . & \Heti{X_{1}X_{3}^{2}X_{4}} \\ 
. & \VR . & \VR . & \VR a_{3} & \VR . & a_{4} & b_{4} & \VR a_{2} & b_{2} & \VR . & . & . & . & . & . & \Heti{X_{2}X_{3}^{2}X_{4}} \\ 
\HR{21} 
. & \VR . & \VR . & \VR . & \VR . & . & . & \VR . & . & \VR a_{1} & . & . & . & . & c_{1} & \Heti{X_{1}^{2}X_{4}^{2}} \\ 
. & \VR . & \VR a_{4} & \VR . & \VR . & . & . & \VR . & . & \VR a_{2} & a_{1} & b_{1} & . & . & c_{2} & \Heti{X_{1}X_{2}X_{4}^{2}} \\ 
. & \VR . & \VR . & \VR . & \VR . & . & . & \VR . & . & \VR . & a_{2} & b_{2} & . & . & c_{5} & \Heti{X_{2}^{2}X_{4}^{2}} \\ 
. & \VR . & \VR . & \VR . & \VR . & . & . & \VR . & . & \VR a_{3} & . & . & a_{1} & b_{1} & c_{3} & \Heti{X_{1}X_{3}X_{4}^{2}} \\ 
. & \VR . & \VR . & \VR a_{4} & \VR . & . & . & \VR . & . & \VR . & a_{3} & b_{3} & a_{2} & b_{2} & c_{6} & \Heti{X_{2}X_{3}X_{4}^{2}} \\ 
. & \VR . & \VR . & \VR . & \VR . & . & . & \VR a_{4} & b_{4} & \VR . & . & . & a_{3} & b_{3} & c_{8} & \Heti{X_{3}^{2}X_{4}^{2}} \\ 
\noalign{\vskip-1pt}
}
$$
$$
W_{1,d}^\circ \ = \ 
\EastBordermatrix{
a_{1} & \VR . & \VR . & \VR . & \VR . & . & . & \VR . & . & \VR . & . & . & . & . & . & \Heti{X_{1}^{4}} \\ 
\HR{20} 
a_{2} & \VR a_{1} & \VR . & \VR . & \VR . & . & . & \VR . & . & \VR . & . & . & . & . & . & \Heti{X_{1}^{3}X_{2}} \\ 
\HR{20} 
a_{3} & \VR . & \VR a_{1} & \VR . & \VR . & . & . & \VR . & . & \VR . & . & . & . & . & . & \Heti{X_{1}^{3}X_{3}} \\ 
\HR{20} 
. & \VR a_{3} & \VR a_{2} & \VR a_{1} & \VR . & . & . & \VR . & . & \VR . & . & . & . & . & . & \Heti{X_{1}^{2}X_{2}X_{3}} \\ 
\HR{20} 
. & \VR a_{2} & \VR . & \VR . & \VR a_{1} & . & . & \VR . & . & \VR . & . & . & . & . & . & \Heti{X_{1}^{2}X_{2}^{2}} \\ 
. & \VR . & \VR . & \VR . & \VR a_{2} & a_{1} & b_{1} & \VR . & . & \VR . & . & . & . & . & . & \Heti{X_{1}X_{2}^{3}} \\ 
. & \VR . & \VR . & \VR . & \VR . & a_{2} & b_{2} & \VR . & . & \VR . & . & . & . & . & . & \Heti{X_{2}^{4}} \\ 
\HR{20} 
. & \VR . & \VR . & \VR a_{2} & \VR a_{3} & . & . & \VR a_{1} & b_{1} & \VR . & . & . & . & . & . & \Heti{X_{1}X_{2}^{2}X_{3}} \\ 
. & \VR . & \VR . & \VR . & \VR . & a_{3} & b_{3} & \VR a_{2} & b_{2} & \VR . & . & . & . & . & . & \Heti{X_{2}^{3}X_{3}} \\ 
\HR{20} 
. & \VR . & \VR a_{3} & \VR . & \VR . & . & . & \VR . & . & \VR a_{1} & . & . & . & . & c_{1} & \Heti{X_{1}^{2}X_{3}^{2}} \\ 
. & \VR . & \VR . & \VR a_{3} & \VR . & . & . & \VR . & . & \VR a_{2} & a_{1} & b_{1} & . & . & c_{2} & \Heti{X_{1}X_{2}X_{3}^{2}} \\ 
. & \VR . & \VR . & \VR . & \VR . & . & . & \VR a_{3} & b_{3} & \VR . & a_{2} & b_{2} & . & . & c_{5} & \Heti{X_{2}^{2}X_{3}^{2}} \\ 
. & \VR . & \VR . & \VR . & \VR . & . & . & \VR . & . & \VR a_{3} & . & . & a_{1} & b_{1} & c_{3} & \Heti{X_{1}X_{3}^{3}} \\ 
. & \VR . & \VR . & \VR . & \VR . & . & . & \VR . & . & \VR . & a_{3} & b_{3} & a_{2} & b_{2} & c_{6} & \Heti{X_{2}X_{3}^{3}} \\ 
. & \VR . & \VR . & \VR . & \VR . & . & . & \VR . & . & \VR . & . & . & a_{3} & b_{3} & c_{8} & \Heti{X_{3}^{4}} \\ 
\noalign{\vskip-1pt}
}
$$
}

Il s'agit maintenant de passer à l'action et de mettre au point tous les ingrédients nécessaires.

\medskip

La première chose à faire pour tenter de prouver l'égalité $\det W_{2,d}^\eq = \det W_{1,d}^\circ$
est de mettre en bijection les bases monomiales de $\Jex_{2,d}^\eq$ et $\Jex_{1,d}^\circ$.
Une idée informelle et plus générale est la suivante. Partons de 
$X^\alpha \in \Jex_{2,d}$ 
et essayons de lui associer un monôme $X^\beta \in \bfA[X_1, \dots, X_{n-1}]_d$ de sorte que $\beta_n = 0$.
On peut procéder au shift vers la gauche des indices des 
indéterminées de $X^\alpha$ comprises entre $\maxDiv(X^\alpha)$ et $n$ ; précisément :
$$
\Shift : \ 
X^\alpha \ \longmapsto \  
X^\beta = 
X_1^{\alpha_1} \cdots X_{j-2}^{\alpha_{j-2}}
X_{j-1}^{\alpha_{j-1}+\alpha_j}X_j^{\alpha_{j+1}} \cdots X_{n-1}^{\alpha_n} 
\qquad 
\hbox{où $j = \maxDiv(X^\alpha)$}
$$
Pour l'instant, on dispose d'une application  
$\Shift : \Jex_{2,d} \to \bfA[X_1, \dots, X_{n-1}]_d$ dont on peut, pour le moment, 
oublier la dépendance en $d$. L'objectif est de montrer que la restriction de $\Shift$ à $\Jex_2^\eq$ 
est injective et de caractériser l'image $\Shift(\Jex_2^\eq)$ (cf.~\ref{CaracJ1sh}). 
En degré $d \geqslant \delta+2$, on va relier $\Jex^\eq_2$ et $\Jex_1^\circ$ en montrant que 
$\Shift(\Jex_{2,d}^\eq) = \Jex_{1,d}^\circ$ 
(cf.~\ref{CaracJ1sh-dGrand}).

\label {NOTA13-Shift}%

\medskip

Avant de donner la définition formelle~\ref{MaxDivShiftDef}, 
expliquons d'une part comment 
retrouver $X^\alpha \in \Jex_2^\eq$ à partir de son image $X^\beta$ par $\Shift$, 
et d'autre part, comment prouver l'inclusion $\Shift(\Jex_2^\eq) \subset \Jex_1^\circ$.

\medskip

\noindent
$\rhd$ 
Prenons $X^\alpha \in \Jex_2^\eq$. On a donc $\alpha_j = d_j$ où $j = \maxDiv(X^\alpha)$.
Le monôme $X^\beta = \Shift(X^\alpha)$ vérifie alors 
$$
\fbox{$
\beta_1 = \alpha_1, 
\qquad \ldots \qquad
\beta_{j-2} = \alpha_{j-2}, \quad
\beta_{j-1} = \alpha_{j-1} + d_j, \quad
\beta_j = \alpha_{j+1}, 
\qquad \ldots \qquad
\beta_{n-1} = \alpha_n, \quad
\beta_n = 0
$}
$$
Pour retrouver $X^\alpha$ à partir de $X^\beta$, il suffit 
d'exprimer l'entier $j$ (qui dépend de $\alpha$) en fonction de $\beta$. 
C'est possible car $\beta_{j-1} \geqslant d_j$ et, pour $k > j$, $\beta_{k-1} = \alpha_k < d_k$. 
Tout cela conduit à 
$$
j = \max \big ( k \in \{2..n\} \mid \beta_{k-1} \geqslant d_k \big )
\qquad \hbox{et} \qquad
X^\alpha = X_1^{\beta_1} \cdots X_{j-2}^{\beta_{j-2}} \ 
X_{j-1}^{\beta_{j-1}-d_j}\  X_j^{d_j}\ X_{j+1}^{\beta_j} \cdots X_{n}^{\beta_{n-1}}
$$

\noindent
$\rhd$ 
Examinons l'image de l'application $\Shift$ en montrant que
$\Shift(\Jex_2^\eq) \subset \Jex_1^\circ$. 
Soit $X^\alpha \in \Jex_2^\eq$ et
introduisons $i = \minDiv(X^\alpha) < j = \maxDiv(X^\alpha)$. 
On distingue les cas $i < j-1$ et $i=j-1$. 
En notant $X^\beta = \Shift(X^\alpha)$, on a :
$$
i < j-1 \ \Rightarrow \ \beta_i = \alpha_i \geqslant d_i
\qquad \hbox{et} \qquad 
i = j-1 \ \Rightarrow \ 
\beta_i = \alpha_i + d_{i+1} \geqslant d_i + d_{i+1}
$$
Visuellement :
$$
\beta = (\dots, \underset{\geqslant d_i}{\beta_i}, \dots, 
\underset{\geqslant d_j}{\beta_{j-1}}, \underset{< d_{j+1}}{\beta_j},
\dots, 
\underset{< d_{n}}{\beta_{n-1}},
0
)
\qquad 
\text{ ou } 
\qquad 
\beta = (\underset{<d_1}{\beta_1} \dots, \underset{<d_{i-1}}{\beta_{i-1}},
\underset{\geqslant d_{i} + d_{i+1}}{\beta_{i}}, \underset{< d_{j+1}}{\beta_{j}},
\dots, 
\underset{< d_{n}}{\beta_{n-1}},
0
)
$$
Dans les deux cas, on constate que $i \in \DivSeq(X^\beta)$, 
donc $X^\beta \in \Jex_1$. 
Et par définition de $X^\beta$, on a $\beta_n = 0$, 
d'où $X^\beta \in \Jex_1^\circ$. 

Au passage, on en déduit que  $i = \minDiv(X^\beta)$, 
c'est-à-dire $\minDiv(X^\alpha) = \minDiv(X^\beta)$.

\bigskip

Résumons. Pour $X^\alpha \in \Jex_2^\eq$, on peut définir le monôme $\Shift(X^\alpha)$
qui habite $\Jex_1^\circ$ 
$$
\begin{array}{rlccccccccc}
\alpha &=& \alpha_1 &\cdots &\alpha_{j-2} &\alpha_{j-1} &d_j &\alpha_{j+1} &\cdots &\alpha_{n-1} &\alpha_n
\\ 
\Shift(\alpha) &=& \alpha_1 &\cdots &\alpha_{j-2} &\alpha_{j-1}+d_j &\alpha_{j+1} &\alpha_{j+2} &\cdots &\alpha_n &0
\\ 
\end{array}
$$
Réciproquement, à partir de $X^\beta \in \Jex_1^\circ$ vérifiant
certaines conditions (cf. proposition~\ref{CaracJ1sh} à venir), on
peut définir un élément de $\Jex_2^\eq$
$$
\begin{array}{rlccccccccc}
\beta &=& \beta_1 &\cdots &\beta_{j-2} &\beta_{j-1} &\beta_j &\beta_{j+1} &\cdots &\beta_{n-1} &0
\\ 
\Shift^{-1}(\beta) &=& \beta_1 &\cdots &\beta_{j-2} &\beta_{j-1}-d_j &d_j &\beta_{j} &\cdots &\beta_{n-2} &\beta_{n-1}
\\ 
\end{array}
$$

\begin{defn}\label{MaxDivShiftDef}
Pour $X^\beta \in \Jex_{1}^\circ$, on pose 
$$
\maxDivSh(X^\beta) \ = \ 
\max \big ( k \in \{2..n\} \mid \beta_{k-1} \geqslant d_k \big ) 
$$
Quand l'ensemble d'indices sur lequel porte le $\max$ est vide, on convient que
celui-ci a une valeur strictement plus petite que 2, par exemple~1. 
Cela se produit si et seulement si $\beta_{k-1} \leqslant d_k - 1$
pour tout $k \geqslant 2$ ;
auquel cas, ${|\beta| \leqslant \delta - (d_1 - 1) \leqslant \delta}$.
Pris à rebours, si $|\beta| \geqslant \delta+1$, l'indice 
$\maxDivSh(X^\beta)$ est dans $\{2..n\}$.
\end{defn}

\label {NOTA13-maxDivSh}%
%
%

\begin{prop} [Caractérisation de l'image $\Shift(\Jex^\eq_2)$ et introduction de $\Jex^\sh_1 \subset \Jex^\circ_1$]
\label{CaracJ1sh} 

\leavevmode
\begin{enumerate}[\rm i)]
\item 
L'application suivante est injective et  conserve $\minDiv$ :
$$
\Shift : \begin{array}[t]{rcl}
\Jex_2^\eq & \longrightarrow & \bfA[\uX] \\ [0.2cm]
X^\alpha & \longmapsto & 
X_1^{\alpha_1} \cdots X_{j-2}^{\alpha_{j-2}} 
X_{j-1}^{\alpha_{j-1}+d_j} X_j^{\alpha_{j+1}} \cdots X_{n-1}^{\alpha_n} 
\quad \text{où $j = \maxDiv(X^\alpha)$}
\end{array}
$$
Son image est incluse dans $\Jex_1^\circ$ et 
pour tout $X^\alpha \in \Jex_{2}^\eq$, on a $\maxDivSh\big(\Shift(X^\alpha)\big) = \maxDiv(X^\alpha)$.

\item 
Notons $\Jex_1^\sh$ le sous-module de $\Jex_1^\circ$ 
de base les $X^\beta$ tels qu'en posant 
$j = \maxDivSh(X^\beta)$ et $i = \minDiv(X^\beta)$, on ait : 
$$
j \in \{2..n\} 
\qquad \text{ et } \qquad 
\left\{
\begin{array}{l} 
i < j-1 \\ 
\text{ou } \\
i=j-1 \hbox{ et } \beta_i \geqslant d_i + d_{i+1}
\end{array}
\right.
$$
Schématiquement
$$
\beta = (\dots, \underset{\geqslant d_i}{\beta_i}, \dots, 
\underset{\geqslant d_j}{\beta_{j-1}}, \underset{< d_{j+1}}{\beta_j},
\dots, 
\underset{< d_{n}}{\beta_{n-1}},
0
)
\quad 
\text{ ou } 
\quad 
\beta = (\underset{<d_1}{\beta_1}, \dots, \underset{<d_{j-2}}{\beta_{j-2}},
\underset{\geqslant d_{j-1} + d_j}{\beta_{j-1}}, \underset{< d_{j+1}}{\beta_j},
\dots, 
\underset{< d_{n}}{\beta_{n-1}},
0
)
$$
On a alors $\boxed{\Jex_1^\sh = \Shift(\Jex_2^\eq)}$.

\item 
La bijection réciproque de $\Shift : \Jex_2^\eq \xrightarrow[]{\simeq} \Jex_1^\sh$ est donnée par
$$
\Shift^{-1} : 
\begin{array}[t]{rcl}
\Jex_1^{\sh} & \longrightarrow & \Jex_2^\eq \\ [0.2cm]
X^\beta & \longmapsto & 
X_1^{\beta_1} \cdots X_{j-2}^{\beta_{j-2}} 
X_{j-1}^{\beta_{j-1}-d_j} X_j^{d_j}X_{j+1}^{\beta_j} \cdots X_{n}^{\beta_{n-1}} 
\quad \text{où $j = \maxDivSh(X^\beta)$}
\end{array}
$$
Dans la correspondance biunivoque monomiale entre $\Jex_2^\eq$ et $\Jex^\sh_1$
donnée par $\Shift$, les primitives $\maxDiv$ et $\maxDivSh$ se correspondent.
\end{enumerate}
\end{prop}

\label {NOTA13-Jex1sh}%
%
%

\begin{proof} \leavevmode

 i)
L'inclusion $\Shift(\Jex^\eq_2) \subset \Jex_1^\circ$ et le fait que
$\Shift$ conserve $\minDiv$ ont été vus en introduction. Et nous avons
été conduits à la définition de $\maxDivSh$ de façon à ce que
$\maxDivSh \circ \Shift = \maxDiv$ sur~$\Jex_2^\eq$.

\smallskip

Passons à l'injectivité.
Soient $X^\alpha$ et $X^{\gamma}$ dans $\Jex_2^\eq$ tels que 
$\Shift(X^\alpha) = \Shift(X^{\gamma})$.
En appliquant $\maxDivSh$, on obtient $\maxDiv(X^\alpha) = \maxDiv(X^\gamma)$, 
entier que l'on note $j$. On a donc  
$$
X_1^{\alpha_1} \cdots X_{j-2}^{\alpha_{j-2}} 
X_{j-1}^{\alpha_{j-1}+d_j} X_j^{\alpha_{j+1}} \cdots X_{n-1}^{\alpha_n} 
\ = \ 
X_1^{\gamma_1} \cdots X_{j-2}^{\gamma_{j-2}} 
X_{j-1}^{\gamma_{j-1}+d_j} X_j^{\gamma_{j+1}} \cdots X_{n-1}^{\gamma_n} 
$$
On en tire $\alpha_\ell = \gamma_\ell$ pour $\ell \neq j$.
D'autre part, $\alpha_j = \gamma_j$, puisque ces deux entiers valent~$d_j$.
D'où $\alpha = \gamma$.

\medskip

ii)
L'inclusion $\Jex_1^\sh \supset \Shift(\Jex_2^\eq)$ a déjà été prouvée
dans l'introduction.  Montrons que $\Jex_1^\sh \subset
\Shift(\Jex_2^\eq)$. Soit $X^\beta \in \Jex_1^\sh$.  En posant $j =
\maxDivSh(X^\beta)$, on propose comme antécédent de $X^\beta$ par
$\Shift$ le monôme
$$
X^\alpha := 
X_1^{\beta_1} \cdots X_{j-2}^{\beta_{j-2}} 
X_{j-1}^{\beta_{j-1}-d_j} X_j^{d_j}X_{j+1}^{\beta_j} \cdots X_{n}^{\beta_{n-1}}
$$
Pour vérifier $X^\alpha \in \Jex_2^\eq$, introduisons
$i = \minDiv(X^\beta)$ et montrons que 
$$
\hbox{\textcircled{\footnotesize 1}} \ 
i, j \in \DivSeq(X^\alpha) \quad \text{avec $i < j$}
\qquad \text{et } \qquad 
\hbox{\textcircled{\footnotesize 2}} \ 
\maxDiv(X^\alpha) = j \quad \text{avec $\alpha_j = d_j$}
$$
Par définition de $X^\alpha$, on a déjà $\alpha_j = d_j$ ;
et par définition de $\maxDivSh(X^\beta)$, on a 
$\alpha_k = \beta_{k-1} < d_k$ pour tout $k > j$.
D'où \textcircled{\footnotesize 2}.
Reste à montrer que $i \in \DivSeq(X^\alpha)$ ; 
ou bien $i < j-1$ auquel cas $\alpha_i = \beta_i \geqslant d_i$ 
ou bien $(i=j-1 \hbox{ et } \beta_i \geqslant d_i + d_{i+1})$ 
auquel cas $\alpha_i = \beta_i - d_{i+1} \geqslant d_i$. 
D'où \textcircled{\footnotesize 1}.

\smallskip
L'égalité $\Shift(X^\alpha) = X^\beta$ résulte de la définition de $\Shift$
(et surtout du fait que $X^\alpha$ a été construit à cet effet!).

\medskip
iii) Résulte de la construction dans le point ii) de l'antécédent $X^\alpha$ de $X^\beta$ par $\Shift$.
\end{proof}

\begin{prop}[L'égalité $\Shift(\Jex_{2,d}^\eq) = \Jex_{1,d}^\circ$ dans le cas $d \geqslant \delta + 2$]\leavevmode
\label{CaracJ1sh-dGrand} 

Pour $d \geqslant \delta +2$, on a $\boxed{\Jex_{1,d}^\sh = \Jex_{1,d}^\circ}$, ce qui
équivaut à dire que $\Shift$ est un isomorphisme monomial:
$$
\Shift : \ 
\Jex_{2,d}^\eq 
\ \xrightarrow{\quad\sim\quad} \ 
\Jex_{1,d}^\circ
$$

\end{prop}

\begin{proof}
Soit $X^\beta \in \Jex_{1,d}^\circ$. Il s'agit de prouver que $X^\beta
\in \Jex^\sh_{1,d}$.  
En posant $i = \minDiv(X^\beta)$ et $j=\maxDivSh(X^\beta)$, on a
$$
\forall\, k \in \{1..i-1\}, \ \, \beta_k < d_k 
\qquad \text{et} \qquad 
\forall\, k \in \{j..n\moins 1\},\ \, \beta_k < d_{k+1}
\leqno(\star)
$$
Comme $d \ge \delta+2$, a fortiori $d \ge \delta+1$, la remarque
incluse dans la définition \ref{MaxDivShiftDef} de $\maxDivSh$ prouve
que $j \in \{2..n\}$

\smallskip

\noindent
$\sbullet$
Montrons que $i < j$. Dans le cas contraire $i \geqslant j$, on aurait $i \in \{j..n-1\}$
et la situation suivante :
$$
d = 
\underbrace{\beta_1 + \dots + \beta_{i-1}}_{
\text{tous ces $\beta_k \leqslant d_{k}-1$}}
+ 
\underbrace{\beta_{i} + \cdots + \beta_{n-1}}_{
\text{tous ces $\beta_k \leqslant d_{k+1}-1$}}
$$
En sommant, on obtiendrait $d \leqslant \delta - (d_i-1)$, a fortiori $d \leqslant \delta$
contredisant $d \geqslant \delta+2$.

\bigskip

\noindent
$\sbullet$
Une fois acquis $i < j$, on est donc dans la situation $i < j-1$ ou $i
= j-1$, situation qui gouverne l'appartenance de $X^\beta$ à $\Jex^\sh_1$.  Le
seul cas à examiner est celui de $i = j-1$. Ecrivons que $d = |\beta|$
et utilisons~$(\star)$, en remplaçant $j$ par $i+1$:
$$
d = 
\underbrace{\beta_1 + \dots + \beta_{i-1}}_{
\text{tous ces $\beta_k \leqslant d_{k}-1$}}
+ \ \beta_i \ + 
\underbrace{\beta_{i+1} + \cdots + \beta_{n-1}}_{
\text{tous ces $\beta_k \leqslant d_{k+1}-1$}}
$$
En sommant, on obtient $d \leqslant \delta - \big((d_i-1) + (d_{i+1}-1)\big) + \beta_i$,
ce qui s'écrit encore :
$$
d \ \leqslant\ \delta + 2 + \big(\beta_i - (d_i + d_{i+1})\big)
$$
La condition $d \geqslant \delta + 2$ impose 
donc $\beta_i \geqslant d_i + d_{i+1}$. 

\bigskip

Résumons. On a, ou bien $i < j-1$ ou bien $i = j-1$ auquel cas
$\beta_i \geqslant d_i + d_{i+1}$.  Donc $X^\beta$ appartient à
$\Jex_{1,d}^\sh$, en vertu de la définition intégrée à la
proposition~\ref{CaracJ1sh}.
\end{proof}

\begin {rmqs} \leavevmode

\medskip
$\bullet$
Pour $d \ge \delta+2$, on peut démontrer que $\Shift(\Jex^\eq_{2,d}) =
\Jex^\circ_{1,d}$, sans passer par $\Jex^\sh_1$, grâce au fait que
$\Shift(\Jex^\eq_{2,d})$ est un sous-module \emph{monomial} de
$\Jex^\circ_{1,d}$ de \emph{même} dimension.  De manière précise:
$$
\dim \Jex^\eq_{2,d}  = \binom{n-2 + d}{n-2} = \dim \Jex^\circ_{1,d}
$$
Nous n'allons pas donner tous les détails. L'égalité de droite résulte
du fait que $\Jex^\circ_{1,d} = \bfA[\uX]^\circ_{d}$ pour $d \ge \delta+1$
et de $\bfA[\uX]^\circ_{d} \simeq \bfA[X_1,\cdots,X_{n-1}]_{d}$ pour tout $d$.

\smallskip

A gauche, c'est plus compliqué.  Montrons d'abord que
$\Jex^\eq_{1,d} = \Jex^\eq_{2,d}$.  En effet, un monôme
$X^\beta \in \Jex^\eq_{1,d}$ vérifie $\beta_j = d_j$ où $j=\maxDiv(X^\beta)$;
puisque $X^\beta/X_j$ est de degré $d-1 \ge \delta+1$, il est
est divisible par un~$X_i^{d_i}$; l'exposant en $j$ de
$X^\beta/X_j$ étant $\beta_j-1 = d_j-1$, on a nécessairement $i < j$ de
sorte que $X^\beta \in \Jex_2$.

\smallskip

Reste à calculer $\dim \Jex^\eq_{1,d}$. On va utiliser à cet effet
les 2 information suivantes, valides pour $d \ge \delta+1$:
$$
\bfA[\uX]_d = \Jex^\ssup_{1,d} \oplus \Jex^\eq_{1,d},
\qquad\qquad
\Jex^\ssup_{1,d} \simeq \Jex_{1,d-1},
\leqno (\star)
$$
l'isomorphie monomiale à droite étant réalisée par les deux bijections
réciproques l'une de l'autre:
$$
\begin{array}[t]{rcl}
\Jex_{1,d-1} & \longrightarrow & \Jex^\ssup_{1,d}
\\ [2mm]
X^\alpha & \longmapsto & X_jX^\alpha,\quad j = \maxDiv(X^\alpha)
\end{array}
\qquad\qquad
\begin{array}[t]{rcl}
\Jex^\ssup_{1,d} & \longrightarrow & \Jex_{1,d-1}
\\ [2mm]
X^\beta & \longmapsto & X^\beta/X_j,\quad j = \maxDiv(X^\beta)
\end{array}
$$
On peut enfin déterminer $\dim\Jex^\eq_{2,d}$ en utilisant $(\star)$
et l'égalité $\Jex_{1,d-1} = \bfA[\uX]_{d-1}$ (due à $d-1 \ge \delta+1$):
$$
\begin {array}{ccl}
\dim\Jex^\eq_{2,d} = \dim\Jex^\eq_{1,d}
                   &=& \dim \bfA[\uX]_d - \dim \Jex^\ssup_{1,d}
\\[2mm]
                  &=& \dim \bfA[\uX]_d - \dim \Jex_{1,d-1}
\\[2mm]
                  &=& \dim \bfA[\uX]_d - \dim \bfA[\uX]_{d-1}
\\[2mm]
&=& \binom{n-1 + d}{n-1} - \binom{n-1 + d-1}{n-1} = \binom{n-2 + d}{n-2}
\\
\end {array}
$$

\medskip
$\bullet$
L'inégalité $d \geqslant \delta+2$ est-elle optimale pour obtenir
$\Shift(\Jex^\eq_{2,d}) = \Jex^\circ_{1,d}$?  Oui, car elle n'est pas
valide pour $d = \delta+1$.  En effet, en degré $\delta +1$, on peut
démontrer que:
$$
\dim \Jex^\eq_{2,\delta+1}
\ = \ 
\binom{n-1+\delta}{n-2} - (n-1) 
\qquad \text{ et } \qquad 
\dim \Jex^\circ_{1,\delta+1} \ = \ \binom{n-1+\delta}{n-2}
$$
L'égalité de droite résulte de
$\Jex^\circ_{1,\delta+1} = \bfA[\uX]^\circ_{\delta+1} \simeq \bfA[X_1,\dots,X_{n-1}]_{\delta+1}$.
Par ailleurs, via $\Shift$, on~a $\Jex^\eq_{2,d} \simeq \Jex^\sh_{1,d}
\subseteq \Jex^\circ_{1,d}$. Et donc pour $d = \delta+1$, l'inclusion
$\Jex^\sh_{1,d} \subseteq \Jex^\circ_{1,d}$ est stricte, mesurée par:
$$
\dim \Jex^\circ_{1,\delta+1}  - \dim \Jex^\sh_{1,\delta+1} = n-1
$$
\end {rmqs}

\subsubsection{Décomposition de $\Jex_1^\sh$ et structure triangulaire de $W_1^\sh(\uP)$}

Jusqu'à présent, sont intervenus des idéaux et isomorphismes
monomiaux, dépendant uniquement du format de degrés~$D$. Maintenant,
nous allons faire intervenir un système $\uP$ de polynômes de format $D$ et les
endomorphismes correspondants $W_\calM(\uP)$ sous-jacents à ces idéaux,
sans mentionner explicitement~$\uP$.

\medskip

Nous allons utiliser la partition des monômes étudiée dans la
section~\ref{sousSectionEndPartition}, partition qui consiste à
regrouper les $X^\alpha \in \Jex_2$ par égalité des $\End(X^\alpha)$
où l'on rappelle que $\End(X^\alpha) = (d_j, \alpha_{j+1}, \dots,
\alpha_n)$ avec $j = \maxDiv(X^\alpha)$.  Pour $\calE \in \scrE =
\scrE(D)$, en appliquant~\ref{EndTriangularite} à $\calM =
\Jex_2^\eq$, on a, avec les notations habituelles
$$
W_2^\eq \big(\Jex_2^\eq\double\calE \big) 
\ \subset \ 
\bigoplus_{\calE' \succcurlyeq \calE}\, \Jex_2^\eq\double{\calE'}
$$ 
Un phénomène triangulaire analogue a lieu du côté de $\Jex_1^\sh$,
à condition de définir correctement le pendant de $\End$.

\begin{defn}\label{DefEndSh}
Pour $X^\beta\in\Jex_1^{\sh}$, on définit $\EndSh(X^\beta) \in \scrE$:
$$
\EndSh(X^\beta) \ = \ 
(d_j, \beta_j, \dots, \beta_{n-1})
\quad \text{où $j = \maxDivSh(X^\beta)$}
$$
Pour $\calE \in \scrE$, on note $\Jex_1^{\sh}\double\calE$ le sous-module de $\Jex_1^{\sh}$ 
de base les monômes $X^\beta$ tels que $\EndSh(X^\beta) = \calE$.
\end{defn}

Le bien-fondé de cette définition résulte des égalités des $(n-j+1)$-uplets :
$$
\begin {array}{l}
\text{pour $X^\alpha \in \Jex_2^{\eq}$,}\quad  
\EndSh\big(\Shift(X^\alpha)\big)
\ = \ 
\End(X^\alpha)
\\[2mm]
\text{pour $X^\beta \in \Jex_1^{\sh}$,} \quad 
\EndSh(X^\beta) \ = \ \End\big(\Shift^{-1}(X^\alpha)\big)
\end {array}
$$

\begin{prop} \label{W1shTriangulaireEnSh}
  
Notons $W_1^\sh = W_1^\sh(\uP)$ l'endomorphisme induit-projeté de $W_1(\uP)$
sur~$\Jex_1^{\sh}$.  Il est triangulaire relativement à la décomposition
$\Jex_1^{\sh} = \bigoplus\limits_{\calE} \Jex_1^{\sh}\double\calE$, 
dans le sens où, pour $\calE \in \scrE(D)$:
$$
W_1^\sh \big(\Jex_1^{\sh}\double\calE \big) 
\ \subset \ 
\bigoplus_{\calE' \succcurlyeq \calE}\, \Jex_1^{\sh}\double{\calE'}
$$ 
\end{prop}

\label {NOTA13-W1sh}%
%
%

\begin{proof} 

Soit $X^\beta \in \Jex_1^{\sh}$. En posant $i = \minDiv(X^\beta)$ et
$j = \maxDivSh(X^\beta)$, l'exposant $\beta$ est, par définition de $\Jex_1^\sh$, dans
l'une des deux configurations suivantes:
$$
\begin {array} {c}
i < j-1 \\ [2mm]
(\dots, \underset{\geqslant d_i}{\beta_i}, \dots, 
\underset{\geqslant d_j}{\beta_{j-1}}, \underset{< d_{j+1}}{\beta_j},
\dots, 
\underset{< d_{n}}{\beta_{n-1}},
0
)
\end {array}
\qquad  \text{ ou }  \qquad
\begin {array} {c}
i = j-1 \\   [2mm]
(\underset{<d_1}{\beta_1}, \dots, \underset{<d_{i-1}}{\beta_{i-1}},
\underset{\geqslant d_{j-1} + d_{j}}{\beta_{j-1}}, \underset{< d_{j+1}}{\beta_{j}},
\dots, 
\underset{< d_{n}}{\beta_{n-1}},
0
)
\end {array}
\leqno (\star)
$$
L'image $W_1^\sh(X^\beta)$ est une combinaison linéaire 
de monômes $X^{\beta'}\in \Jex_1^{\sh}$ 
du type $X^{\beta'} = (X^\beta/X_i^{d_i}) X^\gamma$ 
où~$X^\gamma$ est un monôme de $P_i$. Il s'agit de 
montrer que $\EndSh(X^{\beta}) \preccurlyeq \EndSh(X^{\beta'})$
sachant que:
$$
\EndSh(X^\beta) = (d_j, \beta_j, \dots, \beta_{n-1})
\leqno (\star')
$$
Il va donc falloir cerner $\EndSh(X^{\beta'})$.  Dans la configuration
de gauche où $i < j-1$ (a fortiori $i \ne j-1$), on a $\beta'_{j-1} =
\beta_{j-1} + \gamma_{j-1}$ et dans celle de droite, où $i = j-1$, on
a $\beta'_{j-1} = \beta_{j-1} - d_{j-1} + \gamma_{j-1}$.  
En tenant compte dans~$(\star)$ des inégalités figurant en dessous de
$\beta_{j-1}$, on constate, dans l'une ou l'autre configuration,
que l'on a $\beta'_{j-1} \geqslant d_j$.  Ainsi, en posant $j' :=
\maxDivSh(X^{\beta'})$, on a, par définition de $\maxDivSh$,
l'inégalité $j \le j'$.

\smallskip

\noindent
$\blacktriangleright$ Cas $j < j'$.
D'après la définition~\ref{EndOrder} de la relation d'ordre
$\preccurlyeq$ sur $\scrE$, on a directement l'inégalité $\EndSh(X^{\beta})
\preccurlyeq \EndSh(X^{\beta'})$.

\smallskip

\noindent
$\blacktriangleright$ Cas $j = j'$. Par définition  de $\EndSh$,
on a $\EndSh(X^{\beta'}) = (d_j, \beta'_j, \dots, \beta'_{n-1})$.
En utilisant l'égalité $X^{\beta'} = (X^\beta/X_i^{d_i}) X^\gamma$ et
le fait que $i < j$, on a $\beta'_k = \beta_k + \gamma_k$ pour $k\ge j$,
donc:
$$
\EndSh(X^{\beta'}) = (d_j,\beta_j+\gamma_j,\dots,\beta_{n-1}+\gamma_{n-1})
$$
La comparaison avec $(\star')$ fournit l'inégalité composante à
composante $\EndSh(X^\beta) \preccurlyeq \EndSh(X^{\beta'})$ requise
dans la définition de $\preccurlyeq$.
\end{proof}

\begin{prop} 
Pour tout $\calE \in \scrE$, pour tout entier $d$, 
les endomorphismes $W_{2,d}^\eq\double\calE$ et $W_{1,d}^\sh\double\calE$ 
sont conjugués via l'application 
$\Shift\double\calE : \Jex_{2,d}^\eq\double\calE \to  \Jex_{1,d}^\sh\double\calE$
déduite de $\Shift$,
et ont donc même déterminant.
\end{prop}

\begin {proof}

Définissons l'indice $j$ par $\#\calE = \#[j..n]$.  Commençons par
assurer une remarque préliminaire que nous détaillons: il s'agit de ne
pas trop s'effrayer avec les notations du type
$\Jex_2^\eq\double\calE$.

\medskip
Donnons-nous un $X^\alpha \in \Jex_2^\eq\double\calE$. 
Par définition de~$\EndSh$ (cf.~\ref{DefEndSh}), la bijection $\Shift$ 
fait correspondre $\End$ et $\EndSh$, donc se donner un tel $X^\alpha$ revient à se donner
$X^\beta \in \Jex_1^\sh\double\calE$ via $X^\beta
= \Shift(X^\alpha)$. Notons:
$$
i = \minDiv(X^\alpha) = \minDiv(X^\beta)
$$
l'égalité étant assurée par le fait que $\Shift$ conserve $\minDiv$.
Par définition de~$\Jex_2^\eq\double\calE$, nous avons $j = \maxDiv(X^\alpha)$
donc $i < j$ puisque $X^\alpha \in \Jex_2$.

\medskip
$\bullet$
Considérons un $X^\gamma$ et posons $X^{\alpha'} :=
(X^\alpha/X_i^{d_i}) X^\gamma$, $X^{\beta'} := (X^\beta/X_i^{d_i})
X^\gamma$. Nous affirmons que:

\begin {enumerate}[1.]
\item
Le monôme $X^{\alpha'}$ appartient à $\Jex_2^\eq\double\calE$ si et
seulement si $X^\gamma$ ne dépend que des $j-1$ premières variables.
\item
Idem en ce qui concerne l'appartenance de $X^{\beta'}$ à $\Jex_1^\sh\double\calE$.
\item
Dans ce cas, on a $\Shift(X^{\alpha'}) = X^{\beta'}$.
\end {enumerate}

Pour s'en convaincre, écrivons $X^\alpha$ sous
la forme (préfixe, milieu, suffixe) dans le sens $X^\alpha = m'
X_j^{d_j} m$ où $m',m$ sont des monômes, $m'$ dépendant des $j-1$
premières variables et $m$ des $n-j$ dernières. Par définition de
$\Jex_2^\eq\double\calE$, on a $\End(X^\alpha) = \exposant(X_j^{d_j}m) = \calE$ et
comme $i < j$, on a $i = \minDiv(m')$ permettant d'écrire
$X^\alpha/X_i^{d_i} = (m'/X_i^{d_i}) X_j^{d_j} m$, 
qui est donc un monôme ayant un $\End$-uplet égal à $\calE$.
En multipliant ce monôme par $X^\gamma$, son $\End$-uplet est égal à $\calE$ si et seulement 
si $X^\gamma$ ne dépend que des $j-1$ premières variables.
Cela assure le point~1.
La justification du point~2. est analogue.

\medskip

Concernant le point 3., notons d'abord l'expression générale
$\Shift(X^\alpha) = m' X_{j-1}^{d_j} \widetilde m$ où $\widetilde m$
est le monôme \og décalé\fg{} de $m$. Or nous disposons de 
l'écriture (préfixe, milieu, suffixe) de $X^{\alpha'}$:
$$
X^{\alpha'} = m'' X_j^{d_j} m, \qquad \text{avec} \qquad
m'' = (m'/X_i^{d_i}) X^\gamma
$$
Nous en déduisons $\Shift(X^{\alpha'}) = m'' X_{j-1}^{d_j} \widetilde m$, égal à
$(\Shift(X^\alpha)/X_i^{d_i})\, X^\gamma = X^{\beta'}$, comme annoncé.

\bigskip
$\bullet$
Notons $\calM = \Jex_2^\eq\double\calE$ et $\calN = \Jex_1^\sh\double\calE$.
En gardant les notations ci-dessus, il s'agit d'assurer au niveau de
la flèche pointillée la commutativité du diagramme suivant:
$$
\xymatrix @R = 1.2cm @C = 3cm @M=0.4pc{
X^{\alpha} \ar[d]_{W_\calM} \ar[r]^{\Shift\double\calE} 
& X^\beta \ar[d]^{W_\calN}
\\
\pi_{\calM} \Big( \dfrac{X^{\alpha}}{X_i^{d_i}} \,P_i\Big)
\ar@{-->}[r]^-{\Shift\double\calE} & 
\pi_{\calN} \Big( \dfrac{X^{\beta}}{X_i^{d_i}}\, P_i\Big)
\\
}
$$
\'Ecrivons $P_i$ comme une combinaison de monômes, $P_i = \sum_{|\gamma| = d_i} a_\gamma X^\gamma$,
certains des $a_\gamma$ pouvant être nuls. On a donc:
$$
W_\calM(X^\alpha) = \sum_\gamma a_\gamma (X^\alpha/X_i^{d_i})X^\gamma, \qquad\qquad
W_\calN(X^\beta) = \sum_\gamma a_\gamma (X^\beta/X_i^{d_i})X^\gamma
$$
la somme de gauche étant limitée aux $X^\gamma$ pour lesquels
$X^{\alpha'}:=(X^\alpha/X_i^{d_i})X^\gamma$ appartient à $\calM$ et
celle de droite aux $X^\gamma$ pour lesquels
$X^{\beta'}:=(X^\beta/X_i^{d_i})X^\gamma$ appartient à $\calN$. Ce
qui équivaut, d'après la remarque préliminaire, à limiter chaque somme aux monômes
$X^\gamma$ de $P_i$ ne dépendant que des $j-1$ premières variables.
La commutativité est alors assurée par l'égalité $\Shift(X^{\alpha'}) = X^{\beta'}$
du point 3.
\end {proof}

Nous voici maintenant en mesure de démontrer le résultat annoncé dans l'introduction 
page~\pageref{sectionEgaliteStabiliteMacRae}, à savoir que 
$\dfrac{\det W_{1,d}}{\det W_{2,d}}$ 
est indépendant de $d \geqslant \delta + 1$.

\begin{theo}\label{W2eqVersusW1sh} 
\leavevmode
\begin{enumerate}[\rm i)]
\item Pour $d \in \bbN$, on a $\det W_{2,d}^\eq = \det W_{1,d}^\sh$.
\item 
En particulier, pour $d \geqslant \delta+1$, on a $\det W_{2,d+1}^\eq = \det W_{1,d+1}^\circ$.
\item
On a
$$
\forall\, d,d' \geqslant \delta + 1, \qquad 
\begin {vmatrix}
\det W_{1,d}  & \det W_{1,d'}  \\
\det W_{2,d}  & \det W_{2,d'}  \\  
\end {vmatrix} = 0
$$
\end{enumerate}
\end{theo}

\begin{proof}
i) D'après les structures triangulaires de $W_2^\eq$ et $W_1^\sh$
(cf. \ref{EndTriangularite} et~\ref{W1shTriangulaireEnSh}), 
on a
$$
\det W_{2,d}^\eq \ = \ 
\prod_{\calE} \det W_{2,d}^\eq\double\calE
\qquad 
\text{et} 
\qquad 
\det W_{1,d}^\sh \ = \ 
\prod_{\calE} \det W_{1,d}^\sh\double\calE
$$
En utilisant la proposition précédente, on a $\det W_{2,d}^\eq\double\calE  = 
\det W_{1,d}^\sh\double\calE$, d'où l'égalité annoncée.

ii) Comme $d +1\geqslant \delta+2$, on a $\Jex_{1,d+1}^\sh = \Jex_{1,d+1}^\circ$
d'après~\ref{CaracJ1sh-dGrand}, a fortiori $W_{1,d+1}^\sh = W_{1,d+1}^\circ$
et on conclut avec le point précédent.

iii) Relire le début de la section en page~\pageref{sectionEgaliteStabiliteMacRae} et
utiliser des deux points précédents.
\end{proof}

\begin{exemple}
Prenons un dernier exemple avec $D=(2,3,1,3)$ et $d =\delta=5$, donc
en dessous de~$\delta+2$.  
Nous allons comparer les endomorphismes $W_{2,d}^\eq$ et $W_{1,d}^\sh$.

Les monômes de $\Jex_{2,d}^\eq$ sont regroupés par égalité des $\End$-uplets, 
et ceux de $\Jex_{1,d}^\sh$ par égalité des $\EndSh$-uplets.
Voici les endomorphismes en question. On constate, encore une fois, que leurs déterminants sont égaux !
$$
W_{2,d}^\eq \ = \ 
\EastBordermatrix{
a_{1} & \VR . & . & . & . & . & \VR . & . & . & \VR . & \VR . & . & . & . & \Heti{X_{1}^{2}X_{2}^{3}} \\ 
\HR{17} 
. & \VR a_{1} & . & . & b_{1} & . & \VR . & . & . & \VR . & \VR . & . & . & . & \Heti{X_{1}^{4}X_{3}} \\ 
. & \VR a_{2} & a_{1} & . & b_{2} & b_{1} & \VR . & . & . & \VR . & \VR . & . & . & . & \Heti{X_{1}^{3}X_{2}X_{3}} \\ 
. & \VR a_{5} & a_{2} & a_{1} & b_{5} & b_{2} & \VR . & . & . & \VR . & \VR . & . & . & . & \Heti{X_{1}^{2}X_{2}^{2}X_{3}} \\ 
a_{3} & \VR . & a_{5} & a_{2} & b_{11} & b_{5} & \VR . & . & . & \VR . & \VR . & . & . & . & \Heti{X_{1}X_{2}^{3}X_{3}} \\ 
a_{6} & \VR . & . & a_{5} & . & b_{11} & \VR . & . & . & \VR . & \VR . & . & . & . & \Heti{X_{2}^{4}X_{3}} \\ 
\HR{17} 
. & \VR a_{4} & . & . & b_{4} & . & \VR a_{1} & . & b_{1} & \VR . & \VR . & . & . & . & \Heti{X_{1}^{3}X_{3}X_{4}} \\ 
. & \VR a_{7} & a_{4} & . & b_{7} & b_{4} & \VR a_{2} & a_{1} & b_{2} & \VR . & \VR . & . & . & . & \Heti{X_{1}^{2}X_{2}X_{3}X_{4}} \\ 
a_{9} & \VR . & . & a_{7} & . & b_{13} & \VR . & a_{5} & b_{11} & \VR . & \VR . & . & . & . & \Heti{X_{2}^{3}X_{3}X_{4}} \\ 
\HR{17} 
. & \VR a_{10} & . & . & b_{10} & . & \VR a_{4} & . & b_{4} & \VR a_{1} & \VR . & . & . & . & \Heti{X_{1}^{2}X_{3}X_{4}^{2}} \\ 
\HR{17} 
. & \VR . & . & . & . & . & \VR . & . & . & \VR . & \VR a_{1} & c_{1} & . & . & \Heti{X_{1}^{2}X_{4}^{3}} \\ 
. & \VR . & . & . & b_{20} & . & \VR a_{10} & . & b_{10} & \VR a_{4} & \VR a_{3} & c_{3} & . & c_{1} & \Heti{X_{1}X_{3}X_{4}^{3}} \\ 
. & \VR . & . & . & . & b_{20} & \VR . & a_{10} & b_{16} & \VR a_{7} & \VR a_{6} & . & c_{3} & c_{2} & \Heti{X_{2}X_{3}X_{4}^{3}} \\ 
. & \VR . & . & . & . & . & \VR . & . & b_{19} & \VR a_{9} & \VR a_{8} & . & . & c_{3} & \Heti{X_{3}^{2}X_{4}^{3}} \\ 
\noalign{\vskip-1pt}
}
$$
$$
W_{1,d}^\sh = 
\EastBordermatrix{
a_{1} & \VR . & . & . & . & . & \VR . & . & . & \VR . & \VR . & . & . & . & \Heti{X_{1}^{5}} \\ 
\HR{17} 
a_{2} & \VR a_{1} & . & . & b_{1} & . & \VR . & . & . & \VR . & \VR . & . & . & . & \Heti{X_{1}^{4}X_{2}} \\ 
a_{5} & \VR a_{2} & a_{1} & . & b_{2} & b_{1} & \VR . & . & . & \VR . & \VR . & . & . & . & \Heti{X_{1}^{3}X_{2}^{2}} \\ 
. & \VR a_{5} & a_{2} & a_{1} & b_{5} & b_{2} & \VR . & . & . & \VR . & \VR . & . & . & . & \Heti{X_{1}^{2}X_{2}^{3}} \\ 
. & \VR . & a_{5} & a_{2} & b_{11} & b_{5} & \VR . & . & . & \VR . & \VR . & . & . & . & \Heti{X_{1}X_{2}^{4}} \\ 
. & \VR . & . & a_{5} & . & b_{11} & \VR . & . & . & \VR . & \VR . & . & . & . & \Heti{X_{2}^{5}} \\ 
\HR{17} 
a_{6} & \VR a_{3} & . & . & b_{3} & . & \VR a_{1} & . & b_{1} & \VR . & \VR . & . & . & . & \Heti{X_{1}^{3}X_{2}X_{3}} \\ 
. & \VR a_{6} & a_{3} & . & b_{6} & b_{3} & \VR a_{2} & a_{1} & b_{2} & \VR . & \VR . & . & . & . & \Heti{X_{1}^{2}X_{2}^{2}X_{3}} \\ 
. & \VR . & . & a_{6} & . & b_{12} & \VR . & a_{5} & b_{11} & \VR . & \VR . & . & . & . & \Heti{X_{2}^{4}X_{3}} \\ 
\HR{17} 
. & \VR a_{8} & . & . & b_{8} & . & \VR a_{3} & . & b_{3} & \VR a_{1} & \VR . & . & . & . & \Heti{X_{1}^{2}X_{2}X_{3}^{2}} \\ 
\HR{17} 
. & \VR . & . & . & . & . & \VR . & . & . & \VR . & \VR a_{1} & c_{1} & . & . & \Heti{X_{1}^{2}X_{3}^{3}} \\ 
. & \VR . & . & . & . & . & \VR . & . & . & \VR . & \VR a_{3} & c_{3} & . & c_{1} & \Heti{X_{1}X_{3}^{4}} \\ 
. & \VR . & . & . & . & . & \VR . & . & b_{17} & \VR a_{8} & \VR a_{6} & . & c_{3} & c_{2} & \Heti{X_{2}X_{3}^{4}} \\ 
. & \VR . & . & . & . & . & \VR . & . & . & \VR . & \VR a_{8} & . & . & c_{3} & \Heti{X_{3}^{5}} \\ 
\noalign{\vskip-1pt}
}
$$
Le système $(P_1, P_2, P_3, P_4)$ de format $(2,3,1,3)$ est le suivant ($P_4$ n'intervient pas):

\smallskip

\setlength{\tabcolsep}{2pt}
\noindent
$\left\{
\begin{tabular}{@{}rcp{14cm}} 
$P_{1}$ & $=$ & $a_{1}X_{1}^{2} + a_{2}X_{1}X_{2} + a_{3}X_{1}X_{3} + a_{4}X_{1}X_{4} + 
a_{5}X_{2}^{2} + a_{6}X_{2}X_{3} + a_{7}X_{2}X_{4} + a_{8}X_{3}^{2} + 
a_{9}X_{3}X_{4} + a_{10}X_{4}^{2}$\\ [0.1cm] 
$P_{2}$ & $=$ & $b_{1}X_{1}^{3} + b_{2}X_{1}^{2}X_{2} + b_{3}X_{1}^{2}X_{3} + 
b_{4}X_{1}^{2}X_{4} + b_{5}X_{1}X_{2}^{2} + b_{6}X_{1}X_{2}X_{3} + 
b_{7}X_{1}X_{2}X_{4} + b_{8}X_{1}X_{3}^{2} + b_{9}X_{1}X_{3}X_{4} + 
b_{10}X_{1}X_{4}^{2} + b_{11}X_{2}^{3} + b_{12}X_{2}^{2}X_{3} + 
b_{13}X_{2}^{2}X_{4} + b_{14}X_{2}X_{3}^{2} + b_{15}X_{2}X_{3}X_{4} + 
b_{16}X_{2}X_{4}^{2} + b_{17}X_{3}^{3} + b_{18}X_{3}^{2}X_{4} + 
b_{19}X_{3}X_{4}^{2} + b_{20}X_{4}^{3}$
\\ [0.1cm] 
$P_{3}$ & $=$ & $c_{1}X_{1} + c_{2}X_{2} + c_{3}X_{3} + c_{4}X_{4}$\\ [0.1cm] 
\end{tabular} 
\right.
$
\end{exemple}

\cleardoublepage

{\Large in \verb+14-FFR-ComplexeDecompose.tex+}

\section{Le déterminant de Cayley est un quotient alterné de mineurs}
\label{ComplexeDecompose}


Le contexte est celui de la section~\ref{subsectionCayleyMacRae}, en
particulier du théorème \ref{CayleyDetVectoriel}: un complexe de
Cayley $(F_\sbullet,u_\sbullet)$ de longueur $n$, de caractéristique
d'Euler-Poincaré $c = r_0$, de rangs attendus $(r_k)_{1\leqslant
k \leqslant n+1}$ où $r_{n+1}=0$ et ``sa'' forme déterminant de Cayley
$\mu : \BW^{c}(F_0) \to \bfA$, définie à un inversible près, qui est un
générateur de $\CayleyVect(F_\sbullet)$.

\medskip
Un système d'orientations $(\bfe_0, \dots, \bfe_n)$ de $F_\sbullet$
consiste en la donnée d'une orientation $\bfe_k$ de $F_k$ pour $0 \leqslant
k \leqslant n$.  Dans la suite, on choisit comme isomorphisme déterminantal
$\sharp_k : \bigwedge^{r_{k+1}}(F_k)
\overset{\simeq}{\longrightarrow} \bigwedge^{r_k}(F_k)^\star$
l'isomorphisme de Hodge droit défini par $\bfe_k$ (cf.~\ref{IsoDetOrientation}).
Pour alléger, on utilisera souvent le même symbole~$\sharp$ au lieu de $\sharp_k$.
Le théorème~\ref{Factorisation} permet alors
de définir une suite bien déterminée de vecteurs 
$\Theta_k \in \bigwedge^{r_k}(F_{k-1})$, pour $1 \leqslant k \leqslant n+1$,
factorisant $\bigwedge^{r_k}(u_k)$ sous la forme \og vecteur $\times$ forme linéaire \fg:
$$
\BW^{r_k}(u_k) = \Theta_k\,\Theta_{k+1}^{\sharp_k}, \qquad \qquad \Theta_{n+1} = 1
$$
ainsi que la forme $\mu : \BW^{c}(F_0) \to \bfA$ définie
en~\ref{cLinearFormOfComplex}, qui s'obtient à partir de
$\Theta_1 \in \BW^{r_1}(F_0)$ par $\mu = \Theta_1^{\sharp_0}$.
Nous rappelons (cf. \ref{PassageQuotientFormeAlternee}) que cette
forme passe au quotient sur $\Coker(u_1)$.

\index{isomorphisme déterminantal}%

\medskip
La détermination de $\mu$ (ou de certains de ses coefficients) à laquelle
on s'intéresse ici est grosso-modo basée sur l'obtention par récurrence
de $\Theta_k$ en fonction de $\Theta_{k+1}$:
$$
\Theta_k = \frac{\BW^{r_k}(u_k)(\bfx)} {\Theta_{k+1}^\sharp(\bfx)} \qquad
\text{en faisant le choix de vecteurs \og convenables\fg{} $\bfx \in \BW^{r_k}(F_k)$ }
$$
Malgré une écriture d'aspect compliqué, on ne dit rien d'autre que ce qui suit.
Soit $c \in \bbM_{q,p}(\bfA)$ une matrice produit d'une matrice colonne $b$ par une
matrice ligne $\transpose{a}$, avec $\Gr(\ua) \ge 2$:
$$
\vcenter{
\xymatrix @C=1.5cm{
\bfA^p\ar[dr]|{[a_1,\dots,a_p]} \ar[rr]^{c} && \bfA^q\\
         &\bfA\ar[ur]|{\left[\begin{smallmatrix} b_1\\\vdots\\b_q\end{smallmatrix}\right]} \\
}}
\qquad\qquad
c = b \transpose{a}
$$
On peut alors déterminer $b$ à partir de $(c,a)$ en utilisant
$c_{ij} = b_ia_j$ pour $1\le i \le q$ et $1 \le j \le p$:
$$
b_i = \frac{c_{i1}}{a_1} = \frac{c_{i2}}{a_2} = \cdots = \frac{c_{ip}}{a_p} =
\frac{t_1c_{i1} + t_2c_{i2} + \cdots + t_pc_{ip}}{t_1a_1 + t_2a_2 + \cdots + t_pa_p}
$$

\subsection {Les $\Delta_k$ associés à une décomposition d'un complexe
$F_\sbullet$ de Cayley et l'expression du déterminant de Cayley $\mu : \BW^c(F_0)\to \bfA$ qui s'en déduit}
\label{SousSectionLesDeltak}

On veut expliquer, dans un premier temps de manière informelle,
comment exprimer chaque coefficient de cette forme linéaire~$\mu$
sous la forme d'un quotient alterné de mineurs d'ordre $r_k$~de~$u_k$.
\`A cet effet, on suppose dans cette présentation que chaque terme~$F_k$ du complexe
est muni d'une base ordonnée que l'on prolonge aux puissances extérieures de $F_k$.

Pour une différentielle $u : \big(E,(e_j)\big) \to \big(F,(f_i)\big)$ du complexe, de
rang attendu~$r$, on dispose du régime de croisière suivant :
$$
\vcenter {
\xymatrix @M=0.4pc @R=0.6cm{ 
\bigwedge^{r}(E) \ar[rd]_-{(\Theta')^\sharp} \ar[rr]^-{\bigwedge^{r}(u)} & & 
      \bigwedge^{r}(F) \\
& \bfA \ar[ru]_-{\times \Theta} & 
}}
\qquad \qquad
\left\{
\begin {array} {l}
\BW^{r}(u) \ = \ \Theta \,(\Theta')^\sharp \\[2mm]
\Theta \in \bigwedge^{r}(F), \quad \Theta' \in \bigwedge^{r'}(E) \\[2mm]
r' = \dim E - r \\
\end {array}
\right.
$$
L'isomorphisme déterminantal Hodge droit $\sharp : \bigwedge^{r'}(E)
\overset{\simeq}{\longrightarrow} \bigwedge^{r}(E)^\star$ défini en~\ref{IsoDetOrientation}
est caractérisé par:
$$
\text{pour } \Theta' = \sum_{\#J = r'} \Theta'_J\,e_J,
\qquad 
(\Theta')^\sharp = \sum_{\#J=r'}\Theta'_J\,\varepsilon(J,\overline J)\, e_{\overline J}^\star,
\qquad \text{\idest}\qquad
(\Theta')^\sharp(e_{\overline J}) = \varepsilon\big(J,\overline {J}\big)\,\Theta'_J
$$
où $\overline J$ désigne le
complémentaire de $J$ dans (l'ensemble des indices de) la base de $E$.
Pour $I$ fixé dans (l'ensemble des indices de) la base de $F$ avec
$\#I=r$, on veut récupérer la composante~$\Theta_I$ de $\Theta$ sur
$f_I$.  \'Evaluons l'égalité $\BW^{r}(u)
=\Theta \,(\Theta')^\sharp$ en $e_{\overline J}$ pour $\#J = r'$ puis prenons-en la coordonnée sur $f_I$ :
$$
\textstyle
\bigwedge^{r}(u)(e_{\overline J})
\ = \ 
\Theta \,(\Theta')^\sharp(e_{\overline J})
\qquad \rightsquigarrow \qquad 
\det_{I\times\overline J}(u)
\ = \ 
\varepsilon\big(J,\overline J\big)\,
\Theta_I \,  
\Theta'_J
$$
On fait ainsi apparaître un mineur d'ordre $r$ de $u$ dans le résultat 
(que l'on encadre pour une utilisation ultérieure !) :
$$
\fbox {$
\varepsilon\big(J,\overline J\big)\, \det\nolimits_{I\times\overline J}(u)
\ = \ 
\Theta_I \,  
\Theta'_J
$}
$$
Ce que l'on peut encore écrire de manière symbolique :
$$
\Theta_I \ =\ 
\varepsilon\big(J,\overline J\big)\, \dfrac{\det_{I\times\overline J}(u)}{\Theta'_J}
$$
On peut noter l'asymétrie entre $I$ et $J$~: c'est la composante $\Theta_I$ de 
$\Theta \in \bigwedge^{r}(F)$ sur $f_I$ qui est convoitée (donc $I$ est imposé), 
tandis que $e_J \in \bigwedge^{r'}(E)$ est un vecteur auxiliaire variable
parmi les vecteurs de base.
Inutile de se préoccuper de la possible nullité d'un $\Theta'_J$.
D'une part, car une égalité $a = \frac{b}{c}$ est symbolique et signifie $ac = b$.  
D'autre part, comme $\Gr(\Theta') \geqslant 1$ (et même $\Gr(\Theta') \geqslant 2$), 
on peut \og espérer\fg{} que l'un des~$\Theta'_J$ soit régulier (ce qui est bien mieux que 
\og non nul \fg{}).

\medskip

Aidés de ce régime de croisière pour une différentielle, considérons le complexe $F_\sbullet$
dans son intégralité, le but étant d'obtenir une expression des composantes de $\Theta_1$ (ce qui
revient à obtenir les coefficients de la forme $\mu
= \Theta_1^{\sharp_0}$).  Essayons de cerner la composante sur 
un vecteur de base $e_{I_1}$ de $\bigwedge^{r_1}(F_0)$, où $I_1$ est une partie
de cardinal $r_1$ de la base ordonnée de $F_0$.

Appuyons-nous sur le schéma suivant ($n=5$ ici et $r_6=0$) :
$$
\let\ov=\overline 
\xymatrix @R = 0.2cm {
0\ar[r]  &F_5\ar[r]^{u_5} &F_4\ar[r]^{u_4} &F_3\ar[r]^{u_3} &F_2\ar[r]^{u_2} 
         &F_1\ar[r]^{u_1} &F_0
\\
\text{dimensions}&r_5             &r_5+r_4         &r_4+r_3         &r_3+r_2
         &r_2+r_1         &r_1+r_0
\\
         &\ov{I_6}             &I_5\vee\ov{I_5}     &I_4\vee\ov{I_4}     &I_3\vee\ov{I_3}
         &I_2\vee \ov{I_2}     &*+[o][F]{I_1}
\\
}
$$
Dans ce schéma, $I_2$ est une partie de cardinal $r_2$ de (l'ensemble
des indices de) la base de $F_1$ et $\overline{I_2}$ son
complémentaire (donc de cardinal $r_1$). Idem pour les autres $I_k$.
En vue d'obtenir la composante $(\Theta_1)_{I_1}$ de $\Theta_1$ sur
$e_{I_1}$ (ou encore la valeur de $\mu(e_{\overline{I_1}})$, 
qui diffère de cette dernière valeur par le signe $\varepsilon(I_1, 
\overline{I_1})$),
choisissons une partie $I_2$ de cardinal $r_2$.
En appliquant à $u_1$ ce qui précède (avec
$\bigwedge^{r_1}(u_1) = \Theta_1 \Theta_2^\sharp$), on obtient via
l'encadré de la page précédente:
$$
\varepsilon\big(I_2, \overline {I_2}\big)\, 
\det\nolimits_{I_1\times\overline{I_2}}(u_1)
\ = \ 
(\Theta_1)_{I_1} \,  
(\Theta_2)_{I_2}
$$
L'objectif étant d'attraper la composante de $\Theta_1$ sur $I_1$, 
on réitère ce processus pour exprimer la composante de $\Theta_2$ sur $I_2$:
$$
\varepsilon\big(I_3, \overline {I_3}\big)\, 
\det\nolimits_{I_2\times\overline{I_3}}(u_2)
\ = \ 
(\Theta_2)_{I_2} \,  
(\Theta_3)_{I_3}
$$
Et ainsi de suite jusqu'au bout en se souvenant que $r_6=0$:
$$
\varepsilon\big(I_6, \overline {I_6}\big)\, 
\det\nolimits_{I_5\times\overline{I_6}}(u_5)
\ = \ 
(\Theta_5)_{I_5} \,  
\underbrace{(\Theta_6)_{I_6}}_{= 1}
\qquad \text{avec $\textstyle \Theta_6 = 1 \in \bigwedge^{r_6}(F_5) = \bfA$}
$$
En notant $D_k = \varepsilon\big(I_{k+1}, \overline {I_{k+1}}\big)\,
\det\nolimits_{I_k\times\overline{I_{k+1}}}(u_k)$, mineur d'ordre $r_k$ de $u_k$,
on obtient :
$$
{D_2\ D_4}\, (\Theta_1)_{I_1}
\ =\ 
D_1\ D_3\ D_5
\qquad
\begin {array}{c}
\text{ce que l'on peut écrire} \\
\text {de manière symbolique}  \\
\end {array}
\qquad 
(\Theta_1)_{I_1}
\ =\ 
\dfrac{D_1\ D_3\ D_5}{D_2\ D_4}
$$
Bien sûr, l'idéal serait de faire des choix de $I_2, I_3, \dots$ 
qui fournissent des mineurs non nuls et même réguliers. 
Mais ici, la question n'est pas d'étudier une telle faisabilité
dans un cadre général.

\bigskip

Dans le cadre que nous visons, celui du résultant, les complexes qui
interviendront sont les composantes homogènes de degré $d$ (pour n'importe
quel $d$) du complexe de Koszul de $\uP = (P_1, \ldots, P_n)$ et,
en un certain sens, les $I_k$ seront fixés, donc pas de choix draconien à
réaliser.  Nous pourrons cependant agir au niveau de $\uP$ en
imposant à $\uP$ de couvrir le jeu étalon généralisé ou bien en
générisant~$P_i$ et ceci de deux manières : ou bien complètement via le polynôme
générique homogène de degré~$d_i$ ou bien partiellement via
$P_i + tX_i^{d_i}$ où $t$ est une indéterminée ; chacune de ces stratégies
conduisant au fait que les mineurs concernés sont réguliers.

\smallskip

Par ailleurs, au lieu de manier des mineurs d'ordre $r_k$ de $u_k$
du type $\det_{I_1\times \overline {I_2}}(u_1)$,  $\det_{I_2\times \overline {I_3}}(u_3)$
etc., nous allons opérer de manière différente, plus commode à mettre en place,
en considérant des induits-projetés des $u_k$ relativement à une décomposition
de~$F_\sbullet$. Le bénéfice en est multiple: les bases ordonnées sont
remplacées par des orientations fournissant un mécanisme
plus souple, les mineurs concernés sont normalisés (des scalaires précis
et pas à un inversible près) et enfin et surtout introduction de nouveaux
invariants notés $(\Delta_k)_{1 \leqslant k \leqslant n+1}$.

\subsubsection*{Complexe de Cayley décomposé, version provisoire un tantinet ollé-ollé}

On réalise une passe provisoire et informelle en supposant le complexe
de Cayley $F_\sbullet$ décomposé d'une manière très particulière:
chaque terme $F_k$ est la somme directe de deux sous-modules libres
$F_k = M_k \oplus S_k$ et, pour des raisons pédagogiques, une
hypothèse a priori étrange (voir la suite), $S_k = M_{k-1}$ pour
$k \geqslant 1$. On a donc $r_k = \dim M_{k-1}$ pour $k \geqslant 1$ et $r_0 = \dim S_0$.
Les bases considérées sont adaptées à cette décomposition.

De manière à être en accord avec les notations précédentes, on désigne
par $I_k$ l'ensemble des indices de la base de $F_k$ correspondant aux
vecteurs de $M_{k-1}$, et $\overline {I_k}$ désigne son
complémentaire. On introduit l'induit-projeté~$\beta_k$, construit à
partir de $u_k$, de la manière suivante :
$$
u_k \ =\ 
\NorthEastBordermatrix{ 
\mbox{\footnotesize $M_k$}  & \mbox{\footnotesize $S_k$} & \\
\alpha_k & \cercle{$\beta_k$} & \Heti{M_{k-1}} \\
\noalign{\vskip4pt}
\gamma_k & \delta_k & \Heti{S_{k-1}} \\
}
\qquad
\beta_k = \pi_{M_{k-1}} \circ u_k \circ \iota_{S_k} : S_k \to M_{k-1}
$$
Vu l'hypothèse $S_k = M_{k-1}$, l'application $\beta_k$ est un endomorphisme de $M_{k-1}$ et
$$
\det\nolimits_{I_k \times \overline{I_{k+1}}} (u_k) = \det \beta_k.
$$
C'était le but de cette passe informelle: partant d'une décomposition,
faire apparaître un certain mineur de rang $r_k$ de $u_k$ comme le
déterminant d'un induit-projeté de $u_k$.  En désignant par
$\bfe_{M_0} \in \BW^{r_1}(F_0)$ le produit extérieur de la base de
$M_0$, ce procédé permettrait d'obtenir une expression de la composante de
$\Theta_1 \in \bigwedge^{r_1}(F_0)$ sur $\bfe_{M_0}$ sous la forme
d'un quotient alterné des $\det \beta_k$. Le but final étant,
puisque $\mu = \Theta_1^{\sharp_0}$, d'obtenir une expression des
évaluations de $\mu$ sous forme d'un quotient alterné de déterminants
d'induits-projetés.

\bigskip

Corrigeons une inexactitude : dans le cadre du résultant, les décompositions sont
\emph {presque} de la forme $F_k = M_k \oplus S_k$, sauf qu'il faut remplacer
l'égalité $S_k = M_{k-1}$ par la donnée d'un isomorphisme (de référence) fixé
$\varphi : M_{k-1} \to S_k$.  Ceci conduit à la définition suivante dont le
but principal est de fournir un cadre commun pour les diverses décompositions de
Macaulay de $\rmK_{\sbullet,d}(\uP)$ indexées par $\sigma \in \fS_n$.

\subsubsection*{Complexe de Cayley décomposé}

\begin{defn} \label{ContexteDetCayley}
Une décomposition d'un complexe de Cayley $(F_\sbullet, u_\sbullet)$ de longueur $n$
$$
\xymatrix @M=0.3pc{
0 \ar[r] &
F_n \ar[r]^-{u_n} &
F_{n-1} \ar[r]^-{u_{n-1}} &
\ \cdots \ \ar[r] &
F_1 \ar[r]^-{u_1} &
F_0
}
$$
consiste en la donnée d'une part d'une décomposition de chacun des termes $F_k$:
$$
F_k \ = \ M_k \oplus S_k
\qquad 
\text{avec \ \, $M_k$, $S_k$ libres et $M_n = 0$}
$$
et, d'autre part, pour $k \geqslant 1$, d'un isomorphisme\footnote{%
Il eût été légitime d'indexer~$\varphi$ par $k$, et donc d'utiliser la notation $\varphi_k$.
Pourquoi ne l'avons-nous pas fait ? 
Les auteurs ne le savent pas eux-mêmes.}
$\varphi : M_{k-1} \rightarrow S_k$.

\smallskip
Pour $k \geqslant 1$, on note $\beta_k : S_k \to M_{k-1}$ l'induit-projeté:
$$
u_k \ =\ 
\NorthEastBordermatrix{ 
\mbox{\footnotesize $M_k$}  & \mbox{\footnotesize $S_k$} & \\
\alpha_k & \cercle{$\beta_k$} & \Heti{M_{k-1}} \\
\noalign{\vskip4pt}
\gamma_k & \delta_k & \Heti{S_{k-1}} \\
}
\qquad
\beta_k = \pi_{M_{k-1}} \circ u_k \circ \iota_{S_k}
$$
et $B_k : M_{k-1} \to M_{k-1}$ l'endomorphisme de $M_{k-1}$ défini  par la formule $B_k=\beta_k\circ\varphi$. 
On a donc le schéma
$$
\xymatrix @R=0.1pc @C=4pc @M=0.4pc{
M_n=0 & M_{n-1} \ar@<0.5ex>@{-->}[ddl]^{\varphi} &\cdots  &M_2
  &M_1\ar@<0.5ex>@{-->}[ddl]^{\varphi} &M_0\ar@<0.5ex>@{-->}[ddl]^{\varphi}
\\
\oplus & \oplus & \cdots & \oplus & \oplus & \oplus
\\
S_n \ar@<0.5ex>[uur]^{\beta_n} & S_{n-1} & \cdots & S_2 \ar@<0.5ex>[uur]^{\beta_2} 
& S_1 \ar@<0.5ex>[uur]^{\beta_1} & S_0 \\
}
$$
En posant $r_k = \dim S_k$, on a  $r_k = \dim M_{k-1}$ pour $k \geqslant 1$ et
$r_k + r_{k+1} = \dim F_k$, de sorte que $r_k$ est le rang attendu
de la différentielle $u_k$ du complexe $F_\sbullet$ et $r_0$ la caractéristique
d'Euler-Poincaré de $F_\sbullet$.

\smallskip

Le scalaire $\det B_k$ est bien défini (et pas à un inversible près), ne dépend
que de la décomposition de~$F_\sbullet$ ($\varphi$ étant partie
intégrante de la décomposition) et c'est un mineur d'ordre~$r_k$ de~$u_k$.
\medskip
\end{defn}

En rappelant les notations du chapitre~\ref{ChapStructureMultiplicative} où $\fD_k$ désigne
l'idéal déterminantiel $\calD_{r_k}(u_k)$ et $\fB_k$ l'idéal contenu
$\rmc(\Theta_k)$, nous allons expliciter une égalité de scalaires
$\det B_k = \Delta_k \, \Delta_{k+1}$, compatible avec la
factorisation d'idéaux $\fD_k = \fB_k \,\fB_{k+1}$, compatible au sens
où $\det B_k \in \fD_k$, $\Delta_k \in \fB_k$ et
$\Delta_{k+1} \in \fB_{k+1}$.  Pour définir ces $\Delta_k$, nous
allons avoir besoin des factorisations $\BW^{r_k}(u_k)
= \Theta_k\,\Theta_{k+1}^\sharp$ et de faire intervenir certains
systèmes d'orientations (mais les $\Delta_k$ n'en dépendront pas),
ce qui nécessite quelques définitions précises.

\begin {defn} [Système d'orientations d'une décomposition $(M_\sbullet, S_\sbullet, \varphi)$]
\leavevmode
\begin{enumerate}[\rm i)]
\item
Un système d'orientations de $(M_\sbullet, S_\sbullet, \varphi)$
est la donnée d'orientations $\bfe_{M_k}$ de $M_k$ et $\bfe_{S_k}$ de $S_k$
pour chaque $k \geqslant 0$:
$$
\bfe_{M_k} \in \BW^{r_{k+1}}(M_k) \hookrightarrow \BW^{r_{k+1}}(F_k),\qquad\qquad
\bfe_{S_k} \in \BW^{r_k}(S_k) \hookrightarrow \BW^{r_k}(F_k)
$$
vérifiant $\bfe_{S_k} = \bigwedge^{r_k}(\varphi)(\bfe_{M_{k-1}})$ pour $k \geqslant 1$.

\item
Un système d'orientations de $(M_\sbullet, S_\sbullet, \varphi)$ donne naissance à
un système d'orientations de $F_\sbullet$:
$$
\bfe_{F_k} = \bfe_{M_k} \wedge \bfe_{S_k}  \qquad   k \geqslant 0
$$

\item
Le système d'orientations de $(M_\sbullet, S_\sbullet, \varphi)$ est dit
adapté à un système d'orientations $(\bfe_0, \dots, \bfe_n)$ donné de $F_\sbullet$
si $\bfe_{M_k} \wedge \bfe_{S_k} = \bfe_k$ pour tout $k \geqslant 0$.
\end{enumerate}
\end {defn}

\begin {lem} [Décomposition $\bfe$-orientée]
\label{ExistenceSystemeAdapte}
Soit $\bfe = (\bfe_0, \dots, \bfe_n)$ un système d'orientations
de~$F_\sbullet$ et une décomposition $(M_\sbullet,
S_\sbullet, \varphi)$ de $F_\sbullet$.  Pour toute orientation imposée
$\bfe_{M_0}$ de $M_0$, il existe un unique système d'orientations de
$(M_\sbullet, S_\sbullet, \varphi)$ adapté à $\bfe$.

Nous utiliserons la terminologie ``décomposition $\bfe$-orientée''
pour qualifier la décomposition enrichie par un système d'orientations
adapté à $\bfe$.
\end {lem}

\begin {proof}
Au niveau $F_0$, il y a une et une seule orientation $\bfe_{S_0}$ de $S_0$ telle que
$\bfe_{M_0} \wedge \bfe_{S_0} = \bfe_0$.  Au niveau~$F_1$,
l'orientation~$\bfe_{M_0}$ définit l'orientation $\bfe_{S_1}$; on
considère  alors l'unique orientation $\bfe_{M_1}$ de $M_1$ telle que 
$\bfe_{M_1}\wedge\bfe_{S_1} = \bfe_1$.  Au niveau $F_2$, l'orientation~$\bfe_{M_1}$ définit
l'orientation $\bfe_{S_2}$; on considère alors l'unique orientation
$\bfe_{M_2}$ de $M_2$ telle que $\bfe_{M_2} \wedge \bfe_{S_2} = \bfe_2$. Et ainsi de suite.
\end {proof}

Nous avons déjà vu en~\ref{DependanceOrientation} qu'un système 
de factorisation $(\Theta_k)_{k \geqslant 1}$ de $F_\sbullet$ 
ne dépend que du système d'orientation
$(\bfe_k)_{k \geqslant 1}$ sur $(F_k)_{k \geqslant 1}$ et 
nous avons précisé cette dépendance. Quant à la forme $\mu$,
elle nécessite une orientation $\bfe_0$ sur $F_0$.
Ici, on précise à nouveau cette dépendance lorsque l'on dispose d'une
décomposition $(M_\sbullet, S_\sbullet, \varphi)$ du complexe.

\begin {lem} \label{MultiplicateursOrientationsDecomposition}
Soient deux systèmes d'orientations de $(M_\sbullet, S_\sbullet, \varphi)$,
le second étant noté avec un prime~$'$. Il y a donc des inversibles
$\xi_0, \dots, \xi_n$ et $\xi_{n+1}=1$ tels que $\bfe'_{S_k} = \xi_k\,\bfe_{S_k}$.

\smallskip
Chacun des deux systèmes induit un système d'orientations de $F_\sbullet$, le premier noté
$(\bfe_0, \dots, \bfe_n)$ avec son système associé de vecteurs $(\Theta_k)$,
le second $(\bfe'_0, \dots, \bfe'_n)$ avec son système de vecteurs $(\Theta'_k)$.

\begin {enumerate}[\rm i)]
\item
On a $\Theta'_k = \xi_k\,\Theta_k$ pour $1 \leqslant k \leqslant n+1$.

\item
Soient $\mu, \mu' : \BW^{r_0}(F_0) \to \bfA$ les formes déterminant de
Cayley associées aux deux systèmes:
$$
\mu = \sharp_{\bfe_0}(\Theta_1), \qquad \mu' = \sharp_{\bfe'_0}(\Theta'_1)
$$
Alors $\mu' = \xi_0^{-1}\,\mu$. En particulier, si $\xi_0 = 1$ \idest{} si $\bfe'_{S_0} = \bfe_{S_0}$,
alors $\mu' = \mu$.
\end {enumerate}
\end {lem}

\begin {proof} \leavevmode

i)
On a $\bfe'_k = \varepsilon_k\,\bfe_k$ avec
$\varepsilon_k = \xi_k\xi_{k+1}$ pour $0 \leqslant k \leqslant n$ comme illustré dans le
tableau suivant:
$$
\begin {array}{c|c|c|c|c|c|c|}
k   &n        &n-1       &n-2       &\cdots     &1      &0 \\ [1mm]
\hline  
M_k &\xi_{n+1}=1 &\xi_n        &\xi_{n-1}   &\cdots     &\xi_2    &\xi_1 \\ [1mm]
\hline
S_k &\xi_n       &\xi_{n-1}    &\xi_{n-2}   &\cdots     &\xi_1    &\xi_0  \\ [1mm]
\hline
F_k &\xi_n\xi_{n+1} &\xi_{n-1}\xi_n &\xi_{n-2}\xi_{n-1} &\cdots &\xi_1\xi_2 &\xi_0\xi_1  \\
\end {array}  
$$
D'après la proposition~\ref{DependanceOrientation} (Dépendance en les orientations), on a~:
$$
\let\vep=\varepsilon
\Theta'_k = q_k\,\Theta_k  \qquad \text{avec} \qquad
q_k = \frac {\vep_k \vep_{k+2} \cdots} {\vep_{k+1} \vep_{k+3} \cdots}, \qquad
q_k q_{k+1} = \vep_k, \qquad q_{n+1} = 1.
$$
D'où $q_k q_{k+1} = \xi_k \xi_{k+1}$. 
Comme $q_{n+1}$ et $\xi_{n+1}$ sont tous deux égaux 
à $1$, on en déduit par récurrence que $q_k = \xi_k$.
Bilan: $\Theta'_k = \xi_k \Theta_k$.

\medskip
ii) 
Avec les notations du point i), on a 
$\bfe'_0 = \varepsilon_0 \,\bfe_0$ 
et $\Theta'_1 = \xi_1 \Theta_1$.
D'après la proposition~\ref{DependanceOrientation}, on a 
$
\mu' =
\xi_1 \dfrac{1}{\varepsilon_0} \mu
$.
Comme $\varepsilon_0 = \xi_0 \xi_1$, on en déduit  $\mu' = \dfrac{1}{\xi_0} \mu$.

\end{proof}

\subsubsection*{Notion de coordonnée dans un contexte où il n'y a pas de base}

Soit $L = R \oplus S$ une somme directe où les modules sont pour l'instant
quelconques.  On note $\pi : L \twoheadrightarrow R$ la projection et
$\iota : R \hookrightarrow L$ l'injection canonique si bien que
$\pi \circ \iota = \Id_R$.  Pour $r \in \bbN$, on a
$\BW^r(\pi) \circ \BW^r(\iota) = \Id_{\BW^r(R)}$ de sorte que
$\BW^r(\iota) : \BW^r(R) \to \BW^r(L)$ est une injection. Dans la
suite, dans ce contexte, on identifiera $\BW^r(R)$ à un sous-module de
$\BW^r(L)$ par cette injection.  Si on voit $\pi$ comme un
endomorphisme de $L$, c'est un projecteur d'image $R$, donc
$\BW^r(\pi)$ est un projecteur de $\BW^r(L)$ d'image $\BW^r(R)$, de
sorte que l'on dispose d'une somme directe bien précise $\BW^r(L)
= \BW^r(R) \oplus S'$, à savoir $S' = \Ker \BW^r(\pi)$.

\smallskip

Supposons de plus $R$ libre de rang $r$ et soit $\bfe_R$ une orientation de $R$ \idest{} une
base de $\BW^r(R)$. On voit~$\bfe_R$ à la fois dans $\BW^r(R)$ et dans $\BW^r(L)$.
On dispose donc de la somme directe bien précise:
$$
\BW^r(L) = \bfA.\bfe_R \oplus S'
$$

Dans la suite, dans \emph {ce contexte} i.e. celui d'une décomposition
donnée $L = R \oplus S$ etc., on parlera de \emph{la coordonnée sur
$\bfe_R$ d'un vecteur $\bfx \in \BW^r(L)$}: c'est l'unique scalaire
$a \in \bfA$ tel que $\bfx - a\,\bfe_R \in S'$ ou encore tel que
$\BW^r(\pi)(\bfx) = a\,\bfe_R$.

\smallskip

Supposons de surcroît $S$ libre de rang $s$. Soient $(e_1, \cdots, e_r)$ une base
de $R$ telle que $\bfe_R = e_1 \wedge \cdots \wedge e_r$ et $(e_{r+1}, \cdots, e_{r+s})$ une base
quelconque de $S$. On obtient ainsi une base ordonnée de $L$ et une base $(e_I)_{\#I = r}$
de $\BW^r(L)$ dont $\bfe_R = e_{\{1..r\}}$ fait partie.
On vérifie facilement que
$$
\BW^r(\pi)(e_I) = \begin {cases}
0           &\text{si $I \ne \{1..r\}$}  \\
e_{\{1..r\}} &\text{sinon} \\
\end {cases}
$$
On en déduit que
$$
S' = \bigoplus_{\substack {\#I=r \\ I\ne\{1..r\} \\}} \bfA.e_I
$$
En conséquence, la coordonnée de $\bfx$ sur $\bfe_R$ n'est autre que
la coordonnée de $\bfx$ sur~$\bfe_R$ relativement à la base
$(e_I)_{\#I = r}$. Ceci justifie la terminologie ``coordonnée''
que nous avons utilisée.

\subsubsection*{Où l'on définit enfin les $\Delta_k$}

Pour définir les $\Delta_k$ associés à une décomposition $(M_\sbullet,
S_\sbullet, \varphi)$, on considère un système quelconque
d'orientations $(\bfe_k)_{0\le k\le n}$ de $F_\sbullet$, ce qui donne
naissance aux $(\Theta_k)_{1\le k \le n+1}$ ainsi qu'à $\mu
= \sharp_0(\Theta_1)$. Il est important de comprendre qui dépend de
quoi: les $\Theta_k$ sont indépendants de la décomposition
$(M_\sbullet, S_\sbullet, \varphi)$ tandis que le scalaire $\det B_k$
est défini à partir de cette décomposition sans recours à aucune
orientation.

Un système d'orientations de $(M_\sbullet, S_\sbullet, \varphi)$
adapté à $(\bfe_0, \dots, \bfe_n)$ permet d'exprimer $\det B_k$ à
l'aide des orientations des $M_\sbullet$ et $S_\sbullet$:
$$
\det B_k = \text{la coordonnée sur $\bfe_{M_{k-1}}$ du $r_k$-vecteur 
$\BW^{r_k}(u_k)(\bfe_{S_k})$ de $\BW^{r_k}(F_{k-1})$}
$$
Pour le justifier, utilisons la définition $B_k=\pi\circ u_k \circ \varphi$ où $\pi=\pi_{M_{k-1}}$
et posons $\bfy=\BW^{r_k}(B_k)(\bfe_{M_{k-1}})$. Ce vecteur $\bfy$ est un habitant
de $\BW^{r_k}(M_{k-1})$ qui vérifie $\bfy = \det B_k\,\bfe_{M_{k-1}}$. 
On écrit $\bfy$ de la manière suivante:
$$
\bfy = \BW^{r_k}(\pi)(\bfx) \qquad \text{avec} \qquad
\bfx =  \BW^{r_k}(u_k)\Big(\BW^{r_k}(\varphi)(\bfe_{M_{k-1}})\Big) = \BW^{r_k}(u_k)(\bfe_{S_k})
$$
Cette écriture montre que $\det B_k$ est la coordonnée de $\bfx = \BW^{r_k}(u_k)(\bfe_{S_k})$ sur $\bfe_{M_{k-1}}$.

\begin{prop} [Les $\Delta_k$ associés à une décomposition d'un complexe de Cayley]
\label{SuiteDeltak}
\leavevmode

Dans le contexte ci-dessus, les scalaires $(\Delta_k)_{1 \leqslant k \leqslant n+1}$ sont définis de la manière suivante:
$$
\fbox {$\Delta_k$ est la coordonnée de $\Theta_k$ sur $\bfe_{M_{k-1}}$}
$$
En particulier $\Delta_{n+1} = 1$.

\begin {enumerate}[\rm i)]

\item
Les $\Delta_k$ ne dépendent d'aucun système d'orientations.

\item
On a les égalités scalaires $\Delta_k = \Theta_k^{\sharp}(\bfe_{S_{k-1}})$,
en particulier $\Delta_1 = \mu(\bfe_{S_0})$.

\item
On dispose également des égalités vectorielles $\BW^{r_k}(u_k)(\bfe_{S_k}) = \Theta_k\,\Delta_{k+1}$.

\item
On a en particulier $\BW^{r_1}(u_1)(\bfe_{S_1}) = \Theta_1\,\Delta_2$ et
la factorisation de l'idéal déterminantiel $\calD_{r_1}({u_1}_{|S_1})$
où~${u_1}_{|S_1}$ désigne la restriction de $u_1$ à $S_1$:
$$
\calD_{r_1}({u_1}_{|S_1}) = \Delta_2\,\rmc(\Theta_1)
$$

\item
On dispose des relations $\det B_k = \Delta_k\,\Delta_{k+1}$ pour $1 \leqslant k \leqslant n$.

\item
Les relations précédentes peuvent être écrites de manière symbolique (où $a = \frac{b}{c}$ signifie $a \times c = b$) :
$$
\Delta_k 
\ = \ 
\dfrac{\det B_k \,\det B_{k+2} \, \cdots}
{\det B_{k+1}  \det B_{k+3} \, \cdots}
$$

\item
Supposons tous les $\det B_k$ réguliers. Alors il en est de même des $\Delta_k$
et cette famille $(\Delta_k)_{1 \leqslant k \leqslant n+1}$ est la seule à vérifier
$\det B_k = \Delta_k\,\Delta_{k+1}$ et $\Delta_{n+1}=1$. 
\end{enumerate}
\end{prop}

\index{quotient alterné $\Delta_{k,d}$ en les $(\det B_{k',d})_{k'\ge k}$}%

\begin{proof} \leavevmode

\medskip
i) Soient $(\bfe_0, \dots, \bfe_n)$ et $(\bfe'_0, \dots, \bfe'_n)$
deux systèmes d'orientations de $F_\sbullet$.  Considérons deux
systèmes d'orientations de $(M_\sbullet, S_\sbullet, \varphi)$, le
premier adapté aux $(\bfe_k)$, le second aux $(\bfe'_k)$, noté avec un
prime~$'$.  Utilisons les notations $(\xi_k)$ du
lemme~\ref{MultiplicateursOrientationsDecomposition}, de sorte que
$\bfe'_{S_k} = \xi_k\,\bfe_{S_k}$ donc $\bfe'_{M_{k-1}} = \xi_k\,\bfe_{M_{k-1}}$.
La conclusion du lemme est que $\Theta'_k = \xi_k \Theta_k$.
En conséquence, la coordonnée de $\Theta'_k
= \xi_k\Theta_k$ sur $\bfe'_{M_{k-1}} = \xi_k\,\bfe_{M_{k-1}}$ est
égale à la coordonnée de $\Theta_k$ sur $\bfe_{M_{k-1}}$ \idest{}
$\Delta'_k = \Delta_k$.

\medskip 
ii) Par définition, on a $\Theta_k^{\sharp}(\bfe_{S_{k-1}}) = [\Theta_k \wedge \bfe_{S_{k-1}}]_{\bfe_{k-1}}$ ; 
puisque $\bfe_{k-1} = \bfe_{M_{k-1}} \wedge \bfe_{S_{k-1}}$,
ce scalaire est donc la coordonnée de $\Theta_k$  sur~$\bfe_{M_{k-1}}$
et vaut donc par définition $\Delta_k$.

\medskip
iii)
L'évaluation en $\bfe_{S_k}$ de l'égalité $\BW^{r_k}(u_k) = \Theta_k\, \Theta_{k+1}^\sharp$
fournit dans $\BW^{r_k}(F_{k-1})$:
$$
\BW^{r_k}(u_k)(\bfe_{S_k}) = \Theta_k\,\Delta_{k+1}
$$

\medskip
iv)
En ce qui concerne la factorisation de l'idéal déterminantiel, consisérer les idéaux contenus
dans l'égalité $\BW^{r_1}(u_1)(\bfe_{S_1}) = \Theta_1\,\Delta_2$.

\medskip
v)
La coordonnée sur $\bfe_{M_{k-1}}$ des deux membres de $\BW^{r_k}(u_k)(\bfe_{S_k}) = \Theta_k\,\Delta_{k+1}$
conduit à la relation:
$$
\det B_k = \Delta_k\,\Delta_{k+1}
$$

\medskip
vi) et vii) Ces deux points sont immédiats.
\end{proof}

\medskip

\'Enonçons à présent un résultat très important :
il permet d'obtenir une \textit{expression} de la forme déterminant de
Cayley lorsque l'on dispose d'une décomposition $(M_\sbullet,
S_\sbullet, \varphi)$ de $F_\sbullet$ et du scalaire $\Delta_2$ associé à cette
décomposition (en supposant toutefois que $\Delta_2$ est régulier).
Pour parvenir à cette expression, on fixe un système quelconque~$(\bfe_k)$
d'orientations de $F_\sbullet$, ce qui permet de définir les $\Theta_k$,
la forme $\boxed{\mu = \Theta_1^\sharp}$, générateur de $\CayleyVect(F_\sbullet)$
(cf. le théorème~\ref{CayleyDetVectoriel})
et de considérer un système d'orientations de $(M_\sbullet, S_\sbullet, \varphi)$ adapté
au système d'orientations choisi $(\bfe_k)$ de $F_\sbullet$.

\medskip

Notons $\bfw = \bigwedge^{r_1}(u_1)(\bfe_{S_1})$ et voyons
${\bfw^\sharp : \BW^{c}(F_0) \to \bfA}$ comme une forme $c$-linéaire
alternée $F_0^{c} \to \bfA$.
On rappelle que cette forme s'obtient \og à partir de $r_1$ colonnes (précises) de~$u_1$\fg{}
de la manière suivante.  On considère une base $(e_1, \cdots, e_{r_1})$ de $S_1$
induisant l'orientation~$\bfe_{S_1}$:
$$
\bfe_{S_1} = e_1 \wedge e_2 \wedge \cdots \wedge e_{r_1}
$$
Alors la forme $c$-linéaire alternée $\bfw^\sharp$ s'évalue par
l'intermédiaire d'un déterminant d'ordre $\dim F_0$:
$$
(v_1, \dots, v_{c}) \mapsto
\left[\BW^{r_1}(u_1)(\bfe_{S_1}) \wedge v_1\wedge\cdots\wedge v_{c}\right]_{\bfe_{0}} =
\big[u_1(e_1) \wedge u_1(e_2) \wedge\cdots\wedge u_1(e_{r_1}) \wedge v_1\wedge\cdots\wedge v_{c}\big]_{\bfe_{0}} 
$$
Note: une fois n'est pas coutume mais la petite chanson
\og forme obtenue à partir de $r_1$ colonnes de~$u_1$\fg{} s'adresse aux auteurs,
cf. la section \ref{SectionDetCoRang1}. Et ici,
les colonnes en question sont les images par $u_1$ de la base $(e_i)$ de $S_1$.

\medskip
La proposition suivante nous informe que la forme $c$-linéaire $\mu$  
s'exprime comme le quotient de~$\bfw^\sharp$ par le scalaire $\Delta_2$.  
Dans le cadre du résultant, ce résultat correspond aux égalités 
(cf.~\ref{MacaulayParRecurrenceDegreeDelta} et
\ref{MacaulayParRecurrence}) :
$$
\omegaresP = \frac{\omega_\uP}{\det W_{2,\delta}(\uP)}, \qquad\qquad
\calR_d(\uP) = \frac{\det W_{1,d}(\uP)}{\det W_{2,d}(\uP)}, \qquad  d \geqslant \delta+1
$$
Le lien sera réalisé plus tard en rapprochant~$\Delta_{2,d}(\uP)$ et $\det W_{2,d}(\uP)$:
on montrera que ces deux scalaires sont égaux et ceci pour n'importe quel $d$.

\begin{prop} [Expression d'un générateur du déterminant de Cayley $\CayleyVect(F_\sbullet)$] 
\label{muDelta2Expression}
En utilisant le contexte précédent et ses notations, la forme
$c$-linéaire alternée $\mu := \Theta_1^\sharp$ engendrant $\CayleyVect(F_\sbullet)$
vérifie
$$
\big(\BW^{r_1}(u_1)(\bfe_{S_1})\big)^\sharp = \Delta_2\,\mu
\qquad
\begin {array}{c}
\text{ce que l'on peut écrire} \\
\text{de manière symbolique} \\
\end {array}
\qquad 
\mu = \dfrac{\big(\BW^{r_1}(u_1)(\bfe_{S_1})\big)^\sharp}{\Delta_2}
$$
\end{prop}

\begin{proof}
C'est en fait une redite, \og à $\sharp$-près\fg,  de l'égalité du point iv) de
la proposition~\ref{SuiteDeltak}:
$$
\BW^{r_1}(u_1)(\bfe_{S_1}) = \Theta_1 \Delta_{2} 
$$
à qui on applique~$\sharp = \sharp_0$ pour obtenir celle annoncée.
\end{proof}

On illustrera ce résultat de manière concrète dans la section \ref{CanonicalCayleyDetSection}.

\begin {lem} [Rigidité] \label{RigidityLemma}
Soient $(M_\sbullet, S_\sbullet, \varphi)$ et $(M'_\sbullet, S'_\sbullet, \varphi')$
deux décompositions d'un même complexe de Cayley $F_\sbullet$ telles que $M_0 = M'_0$.
Alors $\Delta_1 = \Delta'_1$.
\end {lem}

\begin {proof} \leavevmode
Indépendamment des décompositions, fixons un système d'orientations
$(\bfe_0, \dots, \bfe_n)$ de $F_\sbullet$, ce qui définit les $\Theta_k$.
Choisissons ensuite un
système d'orientations pour $(M_\sbullet, S_\sbullet, \varphi)$ qui
lui soit adapté.  Idem pour $(M'_\sbullet,
S'_\sbullet, \varphi')$. D'après le
lemme~\ref{ExistenceSystemeAdapte}, on peut imposer $\bfe'_{M'_0}
= \bfe_{M_0}$ (ce qui a du sens puisque $M_0 = M'_0$).

\smallskip
Par définition, $\Delta_k$ est la coordonnée sur $\bfe_{M_{k-1}}$ de $\Theta_k$
tandis que $\Delta'_k$ en est la coordonnée sur $\bfe'_{M'_{k-1}}$.
Pour $k=1$, cela donne $\Delta_1 = \Delta'_1$.
\end {proof}

\begin {prop} [Interprétation de $\Delta_k$ comme déterminant de Cayley] \label{InterpretationStructurelleDeltak}
Dans le contexte d'une décomposition $(M_\sbullet, S_\sbullet, \varphi)$ d'un
complexe de Cayley $F_\sbullet$:

\begin {enumerate} [\rm i)]
\item
Le complexe suivant possède les mêmes rangs attendus que le complexe $F_\sbullet$
$$
\xymatrix @M=0.3pc{
0 \ar[r] &
F_n \ar[r]^-{u_n} &
F_{n-1} \ar[r]^-{u_{n-1}} &
\ \cdots \ \ar[r] &
F_1 \ar[r]^-{u'_1} &
M_0
}
\qquad\qquad
u'_1 = \pi_{M_0} \circ u_1
$$
Sa caractéristique d'Euler-Poincaré est nulle.
Il est de Cayley si et seulement si $\Delta_1$ est régulier auquel cas, $\Delta_1$
est son déterminant de Cayley.

\item
De manière analogue, les rangs attendus du complexe ci-dessous sont dans l'ordre
$(r_n, \dots, r_k)$:
$$
\xymatrix @M=0.3pc{
0 \ar[r] &
F_n \ar[r]^-{u_n} &
F_{n-1} \ar[r]^-{u_{n-1}} &
\ \cdots \ \ar[r] &
F_k \ar[r]^-{u'_k} &
M_{k-1}
}
\qquad\qquad
u'_k = \pi_{M_{k-1}} \circ u_k
$$
Sa caractéristique d'Euler-Poincaré est nulle.
Il est de Cayley si et seulement si $\Delta_k$ est régulier auquel cas, $\Delta_k$
est son déterminant de Cayley.
\end {enumerate}
\end {prop}

\begin {proof} \leavevmode

Les affirmations sur les rangs attendus et la caractéristique d'Euler-Poincaré ne
posent pas de difficulté. On fixe un système d'orientations de $F_\sbullet$ et
un système adapté d'orientations de $(M_\sbullet, S_\sbullet, \varphi)$.

\medskip
i) Identifions $\BW^{r_1}(M_0) = \bfA\bfe_{M_0}$ à $\bfA$ de sorte que
$\BW^{r_1}(\pi_{M_0}) : \BW^{r_1}(F_0) \to \BW^{r_1}(M_0) \simeq \bfA$
est la coordonnée sur $\bfe_{M_0}$. En particulier
$\BW^{r_1}(\pi_{M_0})(\Theta_1) = \Delta_1$ par définition de
$\Delta_1$. Le calcul suivant
$$
\BW^{r_1}(u'_1) = \BW^{r_1}(\pi_{M_0}) \circ \BW^{r_1}(u_1) =
\BW^{r_1}(\pi_{M_0}) \circ (\times \Theta_1) \circ \Theta_2^\sharp =
\BW^{r_1}(\pi_{M_0})(\Theta_1)\,\Theta_2^\sharp  
$$
montre que $\BW^{r_1}(u'_1) = \Delta_1\,\Theta_2^\sharp$.  Comme
$\Gr(\Theta_2) \geqslant 2$, on en déduit que $\Delta_1$ est le pgcd fort de
l'idéal déterminantiel $\calD_{r_1}(u'_1) = \Delta_1\rmc(\Theta_2)$ 
si et seulement si il est régulier.

\medskip
ii)  Il suffit d'appliquer le premier point au complexe tronqué:
$$
\xymatrix @M=0.3pc{
0 \ar[r] &
F_n \ar[r]^-{u_n} &
F_{n-1} \ar[r]^-{u_{n-1}} &
\ \cdots \ \ar[r] &
F_k \ar[r]^-{u_k} &
F_{k-1}
}
$$
\end {proof}


{\Large Dans \verb+14-EnAttente.tex+ car le statut n'est pas clair.}
Et le contenu, écrit il y a des années, n'étant plus à jour par rapport aux chapitres précédents,
doit être examiné et repris.

\subsection{Complexe décomposé : un cas d'école (le degré $d \geqslant \delta+1$)}
 
Voilà ce que dit Chardin dans son résumé de~\cite{Chardin2} 
\og The Resultant via a Koszul Complex \fg{}.

\begin{quote}
{\it 
As noticed by Jouanolou, Hurwitz proved in 1913 \cite{Hurwitz} that,
in the generic case, the Koszul complex is acyclic in positive degrees
if the number of (homogeneous) polynomials is less than or equal to
the number of variables. It was known around 1930 that resultants may
be calculated as a Mc Rae invariant of this complex. This expresses
the resultant as an alternate product of determinants coming from the
differentials of this complex. Demazure explained in a preprint
(\cite{Demazure1}), how to recover this formula from an easy particular case of
deep results of Buchsbaum and Eisenbud on finite free resolutions. He
noticed that one only needs to add one new variable in order to do the
calculation in a non generic situation.

I have never seen any mention of this technique of calculation in
recent reports on the subject (except the quite confidential one of
Demazure and in an extensive work of Jouanolou, however from a rather
different point of view).  So, I will give here elementary and short
proofs of the theorems needed except the well known acyclicity of the
Koszul complex and the ``Principal Theorem of Elimination'' and
present some useful remarks leading to the subsequent algorithm. In
fact, no genericity is needed (it is not the case for all the other
techniques). Furthermore, when the resultant vanishes, some
information can be given about the dimension of the associated
variety.
}
\end{quote}

\medskip
\noindent

``As noticed by Jouanolou'': cf. \cite[4.7]{J3}.
L'auteur y démontre, en suivant Hurwitz, que la suite générique $\uP$ est 1-sécante ; 
or, en terrain gradué, cela entraîne que la suite est complètement sécante.

\bigskip
Et voici maintenant un extrait de la thèse de Penchèvre, \cite[p. 150]{Penchevre}.

\begin{quote}
{\it 
La lignée géométrique des recherches de Cayley montre un décalage par
rapport aux travaux de Bézout, que d'ailleurs Cayley ne cite pas.
A-t-il eu connaissance des idées de Bézout par l'intermédiaire de
l'esquisse qu'en donne Waring dans la seconde édition de ses
Meditationes algebraicae (que Cayley connaissait)?  Bien que cette
esquisse semble aujourd'hui incompréhensible à qui n'a pas lu Bézout,
et qu'ici Cayley n'en fasse pas mention, elle aurait pu au moins lui
donner l'idée. Du reste, Cayley met en oeuvre les mêmes mécanismes
d'algèbre linéaire et rencontre la même pierre d'achoppement (ni
Bézout, ni Cayley ne démontrent que $\Ker f_1 = \Im f_2$). Sa
méthode est exactement la méthode de Bézout, informée par la théorie
des déterminants, la méthode dialytique et le symbolisme matriciel. La
somme alternée $N$ est présente chez Bézout, chez Waring, et chez
Cayley.

Une préoccupation calculatoire chez Cayley constitue un deuxième
décalage par rapport aux travaux de Bézout : Cayley cherche une
formule du résultant. Dans l'article de 1847, il propose la formule
suivante:
$$
R = {Q_1Q_3 \cdots \over Q_2Q_4\cdots}
$$
où pour chaque $i$, $Q_i$ est un mineur de la matrice de $f_i$. Le
choix des mineurs obéit à la règle suivante. La matrice de $f_i$ a le
même nombre de colonnes que la matrice de $f_{i+1}$ compte de lignes.
La règle est que l'ensemble des colonnes sélectionnées pour former le
mineur $Q_i$ doit correspondre au complémentaire de l'ensemble des
lignes sélectionnées pour former le mineur $Q_{i+1}$. Quant au
dernier, soit $Q_j$, il doit être choisi parmi les mineurs maximaux de
la matrice de $f_j$ de sorte qu'aucun des $Q_i$ ne s'annule. Cayley ne
donne pas de règle pour ce choix-ci. Plusieurs auteurs (Salmon, Netto)
semblent avoir essayé de montrer la formule de Cayley, jusqu'à ce
que Macaulay crut bon d'y renoncer, ayant trouvé une formule plus
simple mais \og moins générale\fg{} de la forme $Q_1/\Delta$ où $Q_1$ est
encore un mineur de la matrice de~$f_1$, avec moins de liberté de
choix cependant. En 1926, E.~Fischer est parvenu a démontrer
rigoureusement la formule de Cayley.
}
\end{quote}

\newpage
\subsubsection{Le mécanisme de Cayley pour $n=3$}

Nous souhaitons illustrer les propos rapportés par Penchèvre au sujet 
de la \og préoccupation calculatoire de Cayley \fg{}, et ceci 
en degré $\delta+1$ et pour $n=3$.

$\blacktriangleright$ 
Commençons par traiter le cas d'un complexe libre de longueur $3$ 
de caractéristique d'Euler-Poincaré nulle~:
$$
\xymatrix @M=0.6pc{
0 \ar[r] &
F_2 \simeq \bfA^a \ar[r]^-{u} &
F_1 \simeq \bfA^{a+b} \ar[r]^-{v} &
F_0 \simeq \bfA^b
}
\qquad\qquad
r_2 = a, \quad r_1 = b, \quad r_0 = 0
$$
Le mécanisme de Cayley consiste à choisir une base $\bf1$ de
$\bigwedge^a F_2 \simeq \bfA$, 
à considérer $\Theta = \bigwedge^a(u)({\bf1}) \in \bigwedge^a F_1$ 
puis à créer, à partir de ce vecteur et grâce à un isomorphisme déterminantal $\sharp$, 
une forme linéaire $\Theta^\sharp$.
D'autre part, après l'identification $F_0 \simeq \bfA$, on dispose 
de la forme linéaire $\bigwedge^b v$.
Résumons, nous avons deux formes linéaires :
$$
\textstyle 
\Theta^\sharp : \ \bigwedge\nolimits^b F_1 \longrightarrow \bfA
\qquad \text{ et } \qquad 
\bigwedge\nolimits^b v : \ 
\bigwedge^b F_1 \longrightarrow \bfA
$$
Le théorème de proportionnalité, sous la seule hypothèse 
$v \circ u = 0$, 
nous informe que ces deux formes linéaires sont proportionnelles. 
En notant $(e_i)_{1 \leqslant i \leqslant a+b}$ désigne la base du module central $F_1$, 
cela peut se traduire par les égalités suivantes 
de quotients, à voir comme des produits en croix :
$$
\text{$\forall\  J,\, J'$
parties de cardinal $b$ de $\{1..a+b\}$}, \qquad 
{\bigwedge^b(v)(e_J) \over \Theta^\sharp(e_J)} 
\ = \  
{\bigwedge^b(v)(e_{J'}) \over \Theta^\sharp(e_{J'})}
$$
Choisissons pour 
$\sharp :\ \bigwedge^\sbullet F_1 \rightarrow \Big(\bigwedge^{a+b-\sbullet} F_1\Big)^\star$, 
l'application 
$\Theta \mapsto \Theta^\sharp = [\Theta\wedge\bullet]_\bfe$ où $\bfe$ est une orientation de~$F_1$. 
Pour toute partie $I$ de cardinal $a$, on a alors
$e_I^\sharp = \varepsilon(I,\overline I) {e_{\overline I}}^\star$. 
Avec l'égalité $\Theta = \sum\limits_{\#I = a} 
\det_{I \times \{1..a\}}(u) e_I$, 
on peut alors déterminer $\Theta^\sharp(e_J)$ et l'égalité vectorielle ci-dessus 
s'écrit comme l'égalité de scalaires :
$$
{\det_{\{1..b\} \times J}(v) \over
\varepsilon(\overline J,J) \det_{\overline J\times \{1..a\}}(u)}
\quad = \quad
{\det_{\{1..b\} \times J'}(v) \over
\varepsilon(\overline {J'},J') \det_{\overline J'\times \{1..a\}}(u)}
$$
Ces \og quotients \fg{}, lorsqu'ils sont exacts 
(cf. hypothèse sur la profondeur des idéaux déterminantiels),
valent exactement le déterminant de Cayley du complexe.

Lorsque que l'on fait une hypothèse sur la profondeur des idéaux 
déterminantiels, alors on peut exprimer de manière globale l'égalité 
des quotients ci-dessus.
Précisément, en supposant $\Gr \calD_a(u) \geqslant 2$ \idest{} 
$\Gr(\Theta) \geqslant 2$ \idest{} $\Gr(\Theta^\sharp) \geqslant 2$, 
alors il existe un seul scalaire $q$ tel que $\bigwedge^b(v) = q\, \Theta^\#$. 
Ce scalaire est exactement le déterminant de Cayley du complexe (défini à un inversible près). 
Si on ne veut pas trop détyper les objets, 
cela pourrait être préférable de voir $q$ comme un habitant
$\Theta' \in \bigwedge^b F_0$ avec la factorisation $\bigwedge^b v = \Theta'\, \Theta^\sharp$
$$
\xymatrix @R = 0.4cm {
\bigwedge^b F_1 
\ar[dr]_{\Theta^\sharp}\ar[rr]^{\bigwedge^b v} &       &\bigwedge^b F_0 \simeq \bfA\\
                                        &\bfA\ar[ur]_{\times \Theta'} \\
}                                      
$$

$\blacktriangleright$ 
Appliquons ce qui précéde à la composante homogène de degré $\delta+1$ du complexe de 
Koszul de $(P_1,P_2,P_3)$.
On
rappelle que $\dim \bfA[\uX]_d = \binom{n+d-1}{n-1} \overset{\rm ici}{=} \binom {d+2}{2}$.  
On a
$\rmK_{3,d} \simeq \bfA[\uX]_{d-(d_1+d_2+d_3)} = 0$ 
pour $d < d_1+d_2+d_3 = \delta+3$ \idest{} pour $d \leqslant \delta+2$.  
Le petit complexe qui nous concerne est donc :
$$
\xymatrix @M=0.6pc @C=1.5cm{
0 \ar[r] &
\rmK_{2,\delta+1} \simeq \bfA^{a} \ar[r]^-{\partial_{2,\delta+1}} &
\rmK_{1,\delta+1} \simeq \bfA^{a+b} \ar[r]^-{\Syl_{\delta+1}} &
\rmK_{0,\delta+1} \simeq \bfA^b
}
$$
avec:
$$
b = \binom {\delta+3}{2} = \binom {d_1+d_2+d_3}{2}, \qquad \qquad 
a+b = \binom{d_1+d_2}{2} + \binom {d_1+d_3}{2} + \binom{d_2+d_3}{2}
$$


$\blacktriangleright$ 
Examinons le cas $D = (1,1,2)$

\begin{tabular}{rcp{15cm}} 
$P_{1}$ & $=$ & $a_{1}X_{1} + a_{2}X_{2} + a_{3}X_{3}$\\ [0.1cm] 
$P_{2}$ & $=$ & $b_{1}X_{1} + b_{2}X_{2} + b_{3}X_{3}$\\ [0.1cm] 
$P_{3}$ & $=$ & $c_{1}X_{1}^{2} + c_{2}X_{1}X_{2} + c_{3}X_{1}X_{3} + c_{4}X_{2}^{2} + 
c_{5}X_{2}X_{3} + c_{6}X_{3}^{2}$\\ [0.1cm] 
\end{tabular} 

\noindent
Ici $\delta = 1$ et la résolution est du type $0 \to \bfA \to \bfA^7 \to \bfA^6$:
$$
\partial_{2,\delta+1} \ = \ 
\NorthEastBordermatrix{
\Veti{e_{12}} & \\
-b_{1} & \Heti{X_{1}\,e_{1}} \\
-b_{2} & \Heti{X_{2}\,e_{1}} \\
-b_{3} & \Heti{X_{3}\,e_{1}} \\
a_{1} & \Heti{X_{1}\,e_{2}} \\
a_{2} & \Heti{X_{2}\,e_{2}} \\
a_{3} & \Heti{X_{3}\,e_{2}} \\
. & \Heti{e_{3}} \\
}
\qquad \qquad
\Syl_{\delta+1}\ = \ 
\NorthEastBordermatrix{
\Veti{X_{1}\,e_{1}} & \Veti{X_{2}\,e_{1}} & \Veti{X_{3}\,e_{1}} & \Veti{X_{1}\,e_{2}  \  \leftarrow} & \Veti{X_{2}\,e_{2}} & \Veti{X_{3}\,e_{2}} & \Veti{e_{3}} & \\
a_{1} & . & . & b_{1} & . & . & c_{1} & \Heti{X_{1}^{2}} \\
a_{2} & a_{1} & . & b_{2} & b_{1} & . & c_{2} & \Heti{X_{1}X_{2}} \\
a_{3} & . & a_{1} & b_{3} & . & b_{1} & c_{3} & \Heti{X_{1}X_{3}} \\
. & a_{2} & . & . & b_{2} & . & c_{4} & \Heti{X_{2}^{2}} \\
. & a_{3} & a_{2} & . & b_{3} & b_{2} & c_{5} & \Heti{X_{2}X_{3}} \\
. & . & a_{3} & . & . & b_{3} & c_{6} & \Heti{X_{3}^{2}} \\
}
$$
La structure multiplicative donne une égalité de 
$\binom{a+b}{b} = 7$ quotients de déterminants. 
Voici 4 quotients signés parmi les 7 (qui apparaissent dans l'ordre lexicographique 
des parties $J$ de cardinal $b$ de $\{1..a+b\}$)
$$
\setlength{\arraycolsep}{0.6\arraycolsep}
\def\arraystretch{0.8}
{\left|
\begin{array}{*{6}{c}}
a_{1}& .& .& b_{1}& .& . \\ 
a_{2}& a_{1}& .& b_{2}& b_{1}& . \\ 
a_{3}& .& a_{1}& b_{3}& .& b_{1} \\ 
.& a_{2}& .& .& b_{2}& . \\ 
.& a_{3}& a_{2}& .& b_{3}& b_{2} \\ 
.& .& a_{3}& .& .& b_{3} \\ 
\end{array}
\right|
\over 
1 \times \left|
\begin{array}{*{1}{c}}
0 \\ 
\end{array}
\right|}
= 
{\left|
\begin{array}{*{6}{c}}
a_{1}& .& .& b_{1}& .& c_{1} \\ 
a_{2}& a_{1}& .& b_{2}& b_{1}& c_{2} \\ 
a_{3}& .& a_{1}& b_{3}& .& c_{3} \\ 
.& a_{2}& .& .& b_{2}& c_{4} \\ 
.& a_{3}& a_{2}& .& b_{3}& c_{5} \\ 
.& .& a_{3}& .& .& c_{6} \\ 
\end{array}
\right|
\over 
(-1) \times \left|
\begin{array}{*{1}{c}}
a_3 \\ 
\end{array}
\right|}
= 
{\left|
\begin{array}{*{6}{c}}
a_{1}& .& .& b_{1}& .& c_{1} \\ 
a_{2}& a_{1}& .& b_{2}& .& c_{2} \\ 
a_{3}& .& a_{1}& b_{3}& b_{1}& c_{3} \\ 
.& a_{2}& .& .& .& c_{4} \\ 
.& a_{3}& a_{2}& .& b_{2}& c_{5} \\ 
.& .& a_{3}& .& b_{3}& c_{6} \\ 
\end{array}
\right|
\over 
1 \times \left|
\begin{array}{*{1}{c}}
a_2 \\ 
\end{array}
\right|}
= 
{\left|
\begin{array}{*{6}{c}}
a_{1}& .& .& .& .& c_{1} \\ 
a_{2}& a_{1}& .& b_{1}& .& c_{2} \\ 
a_{3}& .& a_{1}& .& b_{1}& c_{3} \\ 
.& a_{2}& .& b_{2}& .& c_{4} \\ 
.& a_{3}& a_{2}& b_{3}& b_{2}& c_{5} \\ 
.& .& a_{3}& .& b_{3}& c_{6} \\ 
\end{array}
\right|
\over 
(-1) \times \left|
\begin{array}{*{1}{c}}
a_1 \\ 
\end{array}
\right|}
$$
Et les trois derniers qui manquent à l'appel:
$$
\setlength{\arraycolsep}{0.6\arraycolsep}
\def\arraystretch{0.8}
{\left|
\begin{array}{*{6}{c}}
a_{1}& .& b_{1}& .& .& c_{1} \\ 
a_{2}& a_{1}& b_{2}& b_{1}& .& c_{2} \\ 
a_{3}& .& b_{3}& .& b_{1}& c_{3} \\ 
.& a_{2}& .& b_{2}& .& c_{4} \\ 
.& a_{3}& .& b_{3}& b_{2}& c_{5} \\ 
.& .& .& .& b_{3}& c_{6} \\ 
\end{array}
\right|
\over 
1 \times \left|
\begin{array}{*{1}{c}}
-b_3 \\ 
\end{array}
\right|}
= 
{\left|
\begin{array}{*{6}{c}}
a_{1}& .& b_{1}& .& .& c_{1} \\ 
a_{2}& .& b_{2}& b_{1}& .& c_{2} \\ 
a_{3}& a_{1}& b_{3}& .& b_{1}& c_{3} \\ 
.& .& .& b_{2}& .& c_{4} \\ 
.& a_{2}& .& b_{3}& b_{2}& c_{5} \\ 
.& a_{3}& .& .& b_{3}& c_{6} \\ 
\end{array}
\right|
\over 
(-1) \times \left|
\begin{array}{*{1}{c}}
-b_2 \\ 
\end{array}
\right|}
= 
{\left|
\begin{array}{*{6}{c}}
.& .& b_{1}& .& .& c_{1} \\ 
a_{1}& .& b_{2}& b_{1}& .& c_{2} \\ 
.& a_{1}& b_{3}& .& b_{1}& c_{3} \\ 
a_{2}& .& .& b_{2}& .& c_{4} \\ 
a_{3}& a_{2}& .& b_{3}& b_{2}& c_{5} \\ 
.& a_{3}& .& .& b_{3}& c_{6} \\ 
\end{array}
\right|
\over 
1 \times \left|
\begin{array}{*{1}{c}}
-b_1 \\ 
\end{array}
\right|}
$$

Le lecteur attentif aura remarqué qu'en supprimant la 4 eme colonne 
de $\Syl_{\delta+1}$ (\idest{} en considérant $J= \{1,2,3,5,6,7\}$),
on obtient exactement la matrice de l'endomorphisme 
$W_{1,\delta+1}$ (les colonnes sont rangées dans le même ordre que les lignes, 
après identification de $\Smac_{1,\delta+1} \overset{\varphi}{\simeq} \Jex_{1,\delta+1} = 
\bfA[\uX]_{\delta+1}$).
Cette 4\up{ème} colonne est indexée par le seul monôme extérieur dans $\Mmac_{1,\delta+1}$.
Par conséquent, le 4\up{ème} quotient a exactement pour numérateur  $\det W_{1,\delta+1}$.

On remarque que le premier dénominateur est nul; c'est donc que le numérateur est nul.
Et effectivement, l'égalité $\Syl_{\delta+1} \circ \partial_{2,\delta+1} = 0$ 
fournit la relation linéaire entre les 6 premières colonnes $C_i$ de $v = \Syl_{\delta+1}$
(a fortiori entre les colonnes de la matrice numérateur du premier quotient):
$$
-b_1\,C_1  -b_2\,C_2  -b_3\,C_3 + a_1\,C_4 + a_2\,C_5 + a_3\,C_6 \ =\ 0
$$
On remarquera que les autres dénominateurs ne nous laissent pas tomber
et qu'il s'agit des 6 coefficients de $P_1, P_2$.

\textcolor{red} {Dans le texte qui suit, la déf du résultant est celle du
générateur de l'idéal d'élimination et de toutes manières ce texte qui
suit est à revoir.}

Que représente cette valeur commune $\calR$ des quotients ? Réponse : 
c'est le résultant au signe près. Pourquoi ? Par théorème 
(cf. la sous-section~\ref{sousSectionMacaulayRecurrence}, en particulier
le thèorème~\ref{MacaulayParRecurrence} ainsi que le prochain chapitre
Lorsque l'on débute, on apprécie déjà l'égalité de ces $1 + 6$ quotients. 
Et on se dit que l'on devrait être capable, si l'on croit à l'égalité 
$\calR = \pm\Res(P_1, P_2, P_3)$, de prouver \textit{directement} que $\calR \in \ElimIdeal$.
Ce qui revient à montrer (Wiebe étant là pour nous rassurer lorsque $\uP$ est
régulière) que $\langle X_1,X_2,X_3\rangle^{\delta+1} \calR \subset \langle\uP\rangle$ ou encore
sans dénominateur en posant $\calR = \frac{F}G$,
$$
\langle X_1,X_2,X_3\rangle^{\delta+1} \ F 
\quad \overset{?}{\subset} \quad 
G\,\langle P_1, P_2, P_3 \rangle
$$
Mais cette appartenance, c'est une famille (pour $|\alpha| =
\delta+1$) d'identités algébriques du type $X^\alpha F = Q_\alpha \times G
\times (U_{1,\alpha} P_1 + U_{2,\alpha} P_2 + U_{3,\alpha} P_3)$.  
Qui se spécialise et qui est toujours vérifiée, que $\uP$ soit régulière ou pas.
Mystère. En tout cas, lorsque l'on est assez avancé, on peut affirmer
$\calR = -\Res(P_1,P_2,P_3)$: en effet, lorsque l'on spécialise $\uP$
en le jeu étalon \idest{} $a_1 := 1,\, b_2 := 1,\, c_6 := 1$ et les autres coefficients
nuls, le quatrième quotient a pour numérateur
$\det(\Id_6) = 1$ et dénominateur $-1$. 

\bigskip
\noindent
Un mot sur les \textit{endomorphismes} $W_{1,\delta+1}$ et $W_{2,\delta+1}$ :
$$
W_{1,\delta+1} =
\EastBordermatrix{
a_{1} & . & . & . & . & c_{1} & \Heti{X_{1}^{2}} \\ 
a_{2} & a_{1} & . & b_{1} & . & c_{2} & \Heti{X_{1}X_{2}} \\ 
a_{3} & . & a_{1} & . & b_{1} & c_{3} & \Heti{X_{1}X_{3}} \\ 
. & a_{2} & . & b_{2} & . & c_{4} & \Heti{X_{2}^{2}} \\ 
. & a_{3} & a_{2} & b_{3} & b_{2} & c_{5} & \Heti{X_{2}X_{3}} \\ 
. & . & a_{3} & . & b_{3} & c_{6} & \Heti{X_{3}^{2}} \\ 
}
\qquad\qquad
W_{2,\delta+1} = 
\EastBordermatrix{
a_{1} & \Heti{X_{1}X_{2}} \\ 
}
$$
On a fait figurer une base de $\Jex_{1,\delta+1}$ et de $\Jex_{2,\delta+1}$ (réduite à $X_1X_2$).
On rappelle que $W_{1,\delta+1}$ est construite à partir de 
$\Syl_{\delta+1}$ et du mécanisme de sélection $\varphi(X^\alpha) = (X^\alpha/X_i^{d_i})\, e_i$
avec $i=\minDiv(X^\alpha)$. Ici
$$
\varphi(X_1^2) = X_1\,e_1, \quad
\varphi(X_1X_2) = X_2\,e_1, \quad
\varphi(X_1X_3) = X_3\,e_1, \quad
\varphi(X_2X_3) = X_3\,e_2, \quad
\varphi(X_3^2) = e_3
$$

\bigskip

$\blacktriangleright$  
Et pour terminer, l'exemple 
$D = (2,1,2)$.

\medskip

\begin{tabular}{rcp{15cm}} 
$P_{1}$ & $=$ & $a_{1}X_{1}^{2} + a_{2}X_{1}X_{2} + a_{3}X_{1}X_{3} + a_{4}X_{2}^{2} + a_{5}X_{2}X_{3} + a_{6}X_{3}^{2}$\\ [0.1cm] 
$P_{2}$ & $=$ & $b_{1}X_{1} + b_{2}X_{2} + b_{3}X_{3}$\\ [0.1cm] 
$P_{3}$ & $=$ & $c_{1}X_{1}^{2} + c_{2}X_{1}X_{2} + c_{3}X_{1}X_{3} + c_{4}X_{2}^{2} + c_{5}X_{2}X_{3} + c_{6}X_{3}^{2}$\\ [0.1cm] 
\end{tabular}

\smallskip
\noindent
Ici, $\delta=2$, $b = \binom{5}{2} = 10$, $a+b = \binom{3}{2}
+ \binom{3}{2} + \binom{4}{2} = 12$ et la résolution est du
type $0 \to \bfA^2 \to \bfA^{12} \to \bfA^{10}$ avec :
$$
\partial_{2,\delta+1} \ =\  
\NorthEastBordermatrix{
\Veti{e_{12}} & \Veti{e_{23}} & \\
-b_{1} & . & \Heti{X_{1}\,e_{1}} \\
-b_{2} & . & \Heti{X_{2}\,e_{1}} \\
-b_{3} & . & \Heti{X_{3}\,e_{1}} \\
a_{1} & -c_{1} & \Heti{X_{1}^{2}\,e_{2}} \\
a_{2} & -c_{2} & \Heti{X_{1}X_{2}\,e_{2}} \\
a_{3} & -c_{3} & \Heti{X_{1}X_{3}\,e_{2}} \\
a_{4} & -c_{4} & \Heti{X_{2}^{2}\,e_{2}} \\
a_{5} & -c_{5} & \Heti{X_{2}X_{3}\,e_{2}} \\
a_{6} & -c_{6} & \Heti{X_{3}^{2}\,e_{2}} \\
. & b_{1} & \Heti{X_{1}\,e_{3}} \\
. & b_{2} & \Heti{X_{2}\,e_{3}} \\
. & b_{3} & \Heti{X_{3}\,e_{3}} \\
}
\quad
\Syl_{\delta+1} \ = \ 
\NorthEastBordermatrix{
\Veti{X_{1}\,e_{1}} & \Veti{X_{2}\,e_{1}} & \Veti{X_{3}\,e_{1}} & \Veti{X_{1}^{2}\,e_{2}} & \Veti{X_{1}X_{2}\,e_{2}} & \Veti{X_{1}X_{3}\,e_{2}} & \Veti{X_{2}^{2}\,e_{2}} & \Veti{X_{2}X_{3}\,e_{2}} & \Veti{X_{3}^{2}\,e_{2}} & \Veti{X_{1}\,e_{3}} & \Veti{X_{2}\,e_{3}} & \Veti{X_{3}\,e_{3}} & \\
a_{1} & . & . & b_{1} & . & . & . & . & . & c_{1} & . & . & \Heti{X_{1}^{3}} \\
a_{2} & a_{1} & . & b_{2} & b_{1} & . & . & . & . & c_{2} & c_{1} & . & \Heti{X_{1}^{2}X_{2}} \\
a_{3} & . & a_{1} & b_{3} & . & b_{1} & . & . & . & c_{3} & . & c_{1} & \Heti{X_{1}^{2}X_{3}} \\
a_{4} & a_{2} & . & . & b_{2} & . & b_{1} & . & . & c_{4} & c_{2} & . & \Heti{X_{1}X_{2}^{2}} \\
a_{5} & a_{3} & a_{2} & . & b_{3} & b_{2} & . & b_{1} & . & c_{5} & c_{3} & c_{2} & \Heti{X_{1}X_{2}X_{3}} \\
a_{6} & . & a_{3} & . & . & b_{3} & . & . & b_{1} & c_{6} & . & c_{3} & \Heti{X_{1}X_{3}^{2}} \\
. & a_{4} & . & . & . & . & b_{2} & . & . & . & c_{4} & . & \Heti{X_{2}^{3}} \\
. & a_{5} & a_{4} & . & . & . & b_{3} & b_{2} & . & . & c_{5} & c_{4} & \Heti{X_{2}^{2}X_{3}} \\
. & a_{6} & a_{5} & . & . & . & . & b_{3} & b_{2} & . & c_{6} & c_{5} & \Heti{X_{2}X_{3}^{2}} \\
. & . & a_{6} & . & . & . & . & . & b_{3} & . & . & c_{6} & \Heti{X_{3}^{3}} \\
}
$$

Que le lecteur se rassure: on ne va pas lui infliger les $\binom{a+b}{b} = 66$ quotients.
En voici cependant deux: le premier correspond à $J = \{2..11\}$, le second à
$J = \{3..12\}$
$$
\setlength{\arraycolsep}{0.5\arraycolsep}
\def\arraystretch{0.8}
{\left|
\begin{array}{*{10}{c}}
.& .& b_{1}& .& .& .& .& .& c_{1}& . \\ 
a_{1}& .& b_{2}& b_{1}& .& .& .& .& c_{2}& c_{1} \\ 
.& a_{1}& b_{3}& .& b_{1}& .& .& .& c_{3}& . \\ 
a_{2}& .& .& b_{2}& .& b_{1}& .& .& c_{4}& c_{2} \\ 
a_{3}& a_{2}& .& b_{3}& b_{2}& .& b_{1}& .& c_{5}& c_{3} \\ 
.& a_{3}& .& .& b_{3}& .& .& b_{1}& c_{6}& . \\ 
a_{4}& .& .& .& .& b_{2}& .& .& .& c_{4} \\ 
a_{5}& a_{4}& .& .& .& b_{3}& b_{2}& .& .& c_{5} \\ 
a_{6}& a_{5}& .& .& .& .& b_{3}& b_{2}& .& c_{6} \\ 
.& a_{6}& .& .& .& .& .& b_{3}& .& . \\ 
\end{array}
\right|
\over 
1 \times \left|
\begin{array}{*{2}{c}}
-b_{1}& . \\ 
.& b_{3} \\ 
\end{array}
\right|}
\ = \ 
{\left|
\begin{array}{*{10}{c}}
.& b_{1}& .& .& .& .& .& c_{1}& .& . \\ 
.& b_{2}& b_{1}& .& .& .& .& c_{2}& c_{1}& . \\ 
a_{1}& b_{3}& .& b_{1}& .& .& .& c_{3}& .& c_{1} \\ 
.& .& b_{2}& .& b_{1}& .& .& c_{4}& c_{2}& . \\ 
a_{2}& .& b_{3}& b_{2}& .& b_{1}& .& c_{5}& c_{3}& c_{2} \\ 
a_{3}& .& .& b_{3}& .& .& b_{1}& c_{6}& .& c_{3} \\ 
.& .& .& .& b_{2}& .& .& .& c_{4}& . \\ 
a_{4}& .& .& .& b_{3}& b_{2}& .& .& c_{5}& c_{4} \\ 
a_{5}& .& .& .& .& b_{3}& b_{2}& .& c_{6}& c_{5} \\ 
a_{6}& .& .& .& .& .& b_{3}& .& .& c_{6} \\ 
\end{array}
\right|
\over 
1 \times \left|
\begin{array}{*{2}{c}}
-b_{1}& . \\ 
-b_{2}& . \\ 
\end{array}
\right|}
$$
Ces deux quotients ne se laissent pas spécialiser en le jeu étalon
et ne permettent pas de détecter si le quotient commun vaut $\Res(\uP)$ ou $-\Res(\uP)$. 
En revanche, en sélectionnant les colonnes indexées par~$\Smac_{1,\delta+1}$, 
ce qui correspond à $J = \{1..12\} \setminus \{4,11\}$, alors 
on récupère (éventuellement à permutation près des colonnes)  
la matrice de $W_{1,\delta+1}$. On a 
$$
W_{1,\delta+1} 
\ =\ 
\EastBordermatrix{
a_{1} & . & . & . & . & c_{1} & . & . & . & . & \Heti{X_{1}^{3}} \\ 
a_{2} & a_{1} & . & b_{1} & . & c_{2} & . & . & . & . & \Heti{X_{1}^{2}X_{2}} \\ 
a_{3} & . & a_{1} & . & b_{1} & c_{3} & . & . & . & c_{1} & \Heti{X_{1}^{2}X_{3}} \\ 
a_{4} & a_{2} & . & b_{2} & . & c_{4} & b_{1} & . & . & . & \Heti{X_{1}X_{2}^{2}} \\ 
a_{5} & a_{3} & a_{2} & b_{3} & b_{2} & c_{5} & . & b_{1} & . & c_{2} & \Heti{X_{1}X_{2}X_{3}} \\ 
a_{6} & . & a_{3} & . & b_{3} & c_{6} & . & . & b_{1} & c_{3} & \Heti{X_{1}X_{3}^{2}} \\ 
. & a_{4} & . & . & . & . & b_{2} & . & . & . & \Heti{X_{2}^{3}} \\ 
. & a_{5} & a_{4} & . & . & . & b_{3} & b_{2} & . & c_{4} & \Heti{X_{2}^{2}X_{3}} \\ 
. & a_{6} & a_{5} & . & . & . & . & b_{3} & b_{2} & c_{5} & \Heti{X_{2}X_{3}^{2}} \\ 
. & . & a_{6} & . & . & . & . & . & b_{3} & c_{6} & \Heti{X_{3}^{3}} \\ 
}
\qquad
\qquad
W_{2,\delta+1}
\ = \ 
\EastBordermatrix{
a_{1} & . & \Heti{X_{1}^{2}X_{2}} \\ 
a_{6} & b_{2} & \Heti{X_{2}X_{3}^{2}} \\ 
}
$$

$$
B_{2,\delta+1}
\ = \ 
\EastBordermatrix{
a_{1} & -c_{1} & \Heti{X_{1}^{2}\,e_{2}} \\ 
. & b_{2} & \Heti{X_{2}\,e_{3}} \\ 
}
$$
Ici, on a $\Mmac_{1,\delta+1} \simeq \Jex_{2,\delta+1}$ ;
cela se devine sur les bases canoniques indiquées à droite des 
matrices $W_{2,\delta+1}$ et~$B_{2,\delta+1}$.
Elles ne sont pas égales, mais ont même déterminant.




\subsection{La décomposition de Macaulay du complexe $\rmK_\sbullet(\protect\uX^D)$ et l'endomorphisme $B_k(\protect\uP)$}
\label{DecompositionMacaulay}

Par définition même, le $\bfA$-module $\bfB_d$, où $\bfB =
\bfA[\uX]/\langle \uP \rangle$, est présenté par $\Syl_d(\uP)$,
composante homogène de degré $d$ de l'application de Sylvester.  Par
conséquent, un certain nombre d'invariants de ce $\bfA$-module sont
ancrés dans $\Syl_d(\uP)$, par exemple ses idéaux de Fitting, 
via les idéaux déterminantiels de $\Syl_d(\uP)$.  Lorsque
la suite $\uP$ est régulière, c'est \og encore mieux \fg{} car le
$\bfA$-module $\bfB_d$ est librement résoluble, résolu par la
composante homogène de degré $d$ du complexe de Koszul de $\uP$.
Cette résolution libre finie fournit des informations structurelles
supplémentaires sur $\bfB_d$, par exemple le fait qu'il soit de
MacRae de rang $r_{0,d} \overset{\rm def.}{=}\dim \bfA[\uX]_d/\langle X_1^{d_1},
\dots, X_n^{d_n}\rangle_d$.    

Ici, que $\uP$ soit régulière ou pas, il s'agit de faire intervenir
toutes les différentielles des
complexes~$\rmK_{\sbullet,d}(\uP)$ composantes homogènes du complexe
de Koszul $\rmK_\sbullet(\uP)$, complexe dont le terme en degré homologique~$k$
est $\rmK_k = \bigwedge^k\big(\bfA[\uX]^n\big)$.  L'objectif est d'une
part d'exhiber une décomposition monomiale du complexe de Koszul, qui ne
dépend que du format $D = (d_1, \dots, d_n)$ des degrés de $\uP$ et
d'autre part d'en donner les conséquences sur les différentielles, qui
elles sont attachées à~$\uP$.  Nous attirons encore une fois
l'attention sur le fait que certains objets dépendent uniquement du
format $D$, d'autres du système~$\uP$.  La décomposition du complexe
de Koszul, dite de Macaulay, uniquement définie par le jeu étalon
$\uX^D = (X_1^{d_1}, \dots, X_n^{d_n})$, peut être schématisée de la
manière suivante :
$$
\def \egal{\rotatebox{90}{=}}
\xymatrix @H=0pt @R=2pt @C=4pc {
\rmK_n & \rmK_{n-1} &  \cdots  & \rmK_2 & \rmK_1  & \rmK_0  \\
\egal & \egal & & \egal & \egal & \egal \\
\Mmac_n=0 & \Mmac_{n-1} &  \cdots  & \Mmac_2 & \Mmac_1  & \Mmac_0  \\
\oplus & \oplus & & \oplus & \oplus & \oplus \\
\Smac_n \ar@{--}[uur] & \Smac_{n-1} & \cdots & \Smac_2 \ar@{--}[uur]
& \Smac_1 \ar@{--}[uur] & \Smac_0 \\
}
$$
Quelle est la signification du trait en pointillés ? 
Tout ne peut pas être expliqué sur le champ. Mais on peut déjà 
dire que ce trait en pointillés 
\label{FenetreObservation}
a le rôle d'une \og fenêtre d'observation \fg{} sur la différentielle $\partial_k(\uP)$
$$
\partial_k(\uP) \ =\ 
\NorthEastBordermatrix{ 
\mbox{\footnotesize $\Mmac_k$}  & \mbox{\footnotesize $\Smac_k$} & \\
\alpha_k & \cercle{$\beta_k$} & \Heti{\Mmac_{k-1}} \\
\noalign{\vskip4pt}
\gamma_k & \delta_k & \Heti{\Smac_{k-1}} \\
}
\qquad\qquad
\beta_k(\uP) : \Smac_k \to \Mmac_{k-1}
$$
On verra en \ref{varphiIso} que pour $\uP = \uX^D$,
l'induit-projeté $\beta_k(\uX^D)$ est un isomorphisme (monomial) de
degré~$0$.  Comme on peut considérer la composante homogène de degré
$d$ pour n'importe quel $d$, la première conséquence immédiate est
l'égalité dimensionnelle $\dim \Smac_{k,d} = \dim \Mmac_{k-1,d}$ pour
$k \geqslant 1$.
Ainsi la fenêtre d'observation en degré $d$ est \textit{carrée} et, grâce à $\Mmac_{n,d} = 0$,
sa taille est le rang attendu $r_{k,d}$
$$
r_{k,d} \overset{\rm def.}{=}
\dim \rmK_{k,d} - \dim \rmK_{k+1,d} + \cdots + (-1)^{n-k} \dim \rmK_{n,d} =
\dim \Smac_{k,d} = \dim \Mmac_{k-1,d}
$$
Ainsi $\beta_{k,d}(\uP)$ permet la sélection d'un mineur de rang $r_{k,d}$ de la
différentielle $\partial_{k,d}(\uP)$. 
Mais ce n'est pas tout : grâce au fait que $\beta_{k,d}(\uX^D)$ est un
isomorphisme, on pourra donner un sens à $\det \beta_{k,d}(\uP)$ (qui
n'est défini qu'à un inversible près).  Précisément, on va considérer
l'\textit{endomorphisme} de $\Mmac_{k-1,d}$ défini par $B_{k,d}(\uP)
:= \beta_{k,d}(\uP) \circ \beta_{k,d}(\uX^D)^{-1}$ qui vaut l'identité
lorsque l'on spécialise $\uP$ en le jeu étalon~$\uX^D$.

Quand tout sera en place, on obtiendra un fait remarquable : certains
objets initialement définis à l'aide de tout le complexe $\rmK_{\sbullet,d}(\uP)$
et qui semblent nécessiter des mineurs ad hoc de \textit{toutes} les
différentielles, pourront en fait s'exprimer uniquement en termes de
la matrice de Sylvester $\Syl_d(\uP)$, c'est-à-dire uniquement à
partir de la \textit{première} différentielle $\partial_{1,d}(\uP)$,
et ceci grâce aux endomorphismes~$W_{h,d}(\uP)$.  C'est l'objet du
chapitre~\ref{ChapBW}, en particulier du
théorème~\ref{DetBkdProdBinomialDetWhd} dont voici un cas particulier:
$$
\det B_{2,d} = \det W_{2,d}\ \det W_{3,d}\ \cdots\ \det W_{n,d}
$$

\label {NOTA14-partialkP}%
%
%

\begin{rmq} \label{FenetreObservationPourSyl}
  
Avant les définitions formelles, on peut préciser ce que sont
$\Mmac_0, \Smac_0, \Mmac_1, \Smac_1$. Commençons par $k=0$: $\Mmac_0$
est tout simplement l'idéal monomial $\Jex_1 = \langle X_1^{d_1},
\dots, X_n^{d_n}\rangle$ et $\Smac_0$ son supplémentaire
monomial.  En ce qui concerne~$\Smac_1$, on peut le faire apparaître
via $\Smac_1 = \varphi(\Jex_1)$ où $\varphi : \Jex_1 \rightarrow
\rmK_1$ est l'injection monomiale définie par $X^\alpha \mapsto
(X^\alpha/X_i^{d_i})\,e_i$ avec $i = \minDiv(X^\alpha)$, cas
particulier de la définition ultérieure~\ref{varphiIso}.  Enfin
$\Mmac_1$ est le sous-$\bfA$-module de $\rmK_1$ de base les monômes
extérieurs n'appartenant pas à $\Smac_1$; ce qui n'est pas banal,
c'est que $\Mmac_1$ est un sous-$\bfA[\uX]$-module de type fini comme
on le verra plus loin.

La \og fenêtre d'observation\fg{} $\Mmac_0 \leftrightarrow \Smac_1$
peut être utilisée pour donner une définition de l'endomorphisme
$W_\calM(\uP)$, où $\calM$ est un sous-module monomial de $\Jex_1$,
uniquement à partir de l'application de Sylvester et de l'injection
monomiale $\varphi$ restreinte à $\calM$:
$$
W_\calM(\uP) = \pi_{\calM} \circ \Syl(\uP) \circ \varphi_{|\calM}
$$
\end{rmq}

\subsubsection{$\bullet$ 
Le sous-$\bfA[\protect\uX]$-module $\Mmac_k \subset \rmK_k$ attaché au format $D$ et
      son $\bfA$-supplémentaire monomial $\Smac_k$}

Nous présentons la décomposition monomiale de l'algèbre extérieure
$\bigwedge\big(\bfA[\uX]^n\big)$ en tant que $\bfA$-module, due à
Macaulay, et qui figure chez~\cite{Demazure1}.  La qualification \og monomiale\fg{} prend
sa source dans la présence de la $\bfA$-base monomiale de $\bigwedge\big(\bfA[\uX]^n\big)$ constituée des
monômes extérieurs $X^\beta\,e_J$. Cette décomposition,
$\rmK_k = \Mmac_k\oplus\Smac_k$, a deux vertus.  Tout d'abord, lorsque
$k=0$, elle généralise à l'algèbre extérieure la décomposition monomiale suivante
de $\rmK_0 = \bfA[\uX]$:
$$
\bfA[\uX] \ = \  \Mmac_0\oplus\Smac_0 \ = \ 
\Jex_1
\ \oplus \, 
\bigoplus_{\alpha \preccurlyeq \emouton} \! \bfA X^\alpha
$$
décomposition qui partage les monômes en deux catégories selon que l'ensemble 
de $D$-divisibilité du monôme est non vide ou pas.
De plus, cette décomposition possède la propriété suivante:
la première fenêtre d'observation, indexée par $\Smac_1$ et $\Mmac_0 = \Jex_1$ 
(donc indexée par~$\Mmac_0 = \Jex_1$ après composition avec l'isomorphisme~$\varphi$
de la remarque précédente), s'identifie à l'endomorphisme~$W_1$ de $\Jex_1$.

Cette décomposition est pilotée par le mécanisme de sélection $\minDiv
: \{\text{monômes}(\bfA[\uX])\} \rightarrow \bbN$ attaché au format
$D=(d_1, \dots, d_n)$.
Nous en verrons ultérieurement une version
tordue par $\sigma \in \fS_n$, pilotée par $\sminDiv$.
On rappelle la convention utilisée dans la définition~\ref{NOTA05-minDiv} de $\minDiv$,
à savoir que $\min J$, pour une partie $J \subset \{1..n\}$, est toujours
défini même si $J$ est vide: $\min\emptyset = n+1$. En particulier,
$\minDiv(X^\beta) = n+1$ lorsque $\DivSeq(X^\beta) = \emptyset$.

\medskip

Dans la définition ci-dessous, le prédicat $i < J$, pour $1 \leqslant i\le
n$ et $J \subset \{1..n\}$, a la signification \og $i$ est strictement
plus petit que tout élément de $J$\fg. Il est donc vérifié si $J =
\emptyset$  et avec la convention ci-dessus le prédicat en question
peut être écrit $i < \min J$.

\index{de@décomposition de Macaulay de $\rmK_\sbullet(\uP)$}%

\begin{defns} [Sous-module de Macaulay $\Mmac_k$ et son supplémentaire monomial $\Smac_k$]
\label{DefDecompositionMacaulay}
\leavevmode
  
\begin {enumerate} [\rm i)]
\item
On note $\Mmac_k$ le sous-$\bfA$-module de $\rmK_k = \bigwedge^k(\bfA[\uX]^n)$ 
de base les~$X^\beta e_J$ vérifiant $X^\beta \in \langle X_i^{d_i} \mid \ i < J \rangle$ et
$\#J = k$.
En tenant compte des conventions rappelées avant la définition,
la contrainte sur $X^\beta$ peut être écrite $\minDiv(X^\beta) < \min J$,
cette inégalité stricte impliquant $X^\beta \in \Jex_1$.

Ce sous-module $\Mmac_k$ est un sous-$\bfA[\uX]$-module de $\rmK_k$, nommé
sous-module de Macaulay de $\rmK_k$. C'est le sous-$\bfA[\uX]$-module de $\rmK_k$
engendré par les $X_i^{d_i}\,e_J$  avec $i < J$ et $\#J = k$.

\item
On désigne par $\Smac_k$ le supplémentaire monomial de $\Mmac_k$ dans $\rmK_k$. 
C'est donc le sous-$\bfA$-module de $\rmK_k$ de base les monômes extérieurs
non dans $\Mmac_k$.

\item
En résumé, sans recourir à $\minDiv$:
$$
\textstyle
\bigwedge^k(\bfA[\uX]^n) \ = \ \Mmac_k \oplus \Smac_k
\qquad 
\text{avec }
\qquad 
X^\beta e_J \ \in \ 
\left\{
\begin{array}{ll}
\Mmac_k & \text{si $\exists\, i < \min J$ tel que $\beta_i \geqslant d_i$}\\ [0.3em]
\Smac_k & \text{si $\forall \, i < \min J$, $\beta_i < d_i$}
\end{array}
\right.
$$
Pour $k=0$, on retrouve la décomposition 
$\bfA[\uX] =\Mmac_0 \oplus \Smac_0$
avec $\Mmac_0 = \Jex_1 = \langle X_1^{d_1}, \dots, X_n^{d_n}\rangle$ 
et $\Smac_0 = \bigoplus_{\alpha \preccurlyeq \emouton} \bfA X^\alpha$.
Pour $k=n$, on a $J = \{1..n\}$, à fortiori $\Mmac_n = 0$.
\end {enumerate}
\end{defns}

\label {NOTA14-MmacSmac}%

\bigskip
Voilà une propriété concernant la différentielle du complexe de Koszul $\rmK_\sbullet$ de $\uP$,
propriété qui n'a pas été exploitée par les auteurs.

\begin{prop}\label{CuriositeMcomplexe}
On a $\partial(\uP)(\Mmac_k) \subset \Mmac_{k-1}$. 
Ainsi $\Mmac_\sbullet$ est 
un sous-complexe du complexe~$\rmK_\sbullet(\uP)$.
\end{prop}

\begin{proof}
Soit $X^\alpha e_I \in \Mmac_k$. 
Pour $i \in I$, on a l'implication 
évidente~\mbox{$j < \min I \Rightarrow j < \min (I \setminus i)$.}
Par conséquent,
on obtient $X^\alpha e_{I \setminus i} \in \Mmac_{k-1}$. 
Comme $\Mmac_{k-1}$ est un $\bfA[\uX]$-module, on a 
$P_i X^\alpha  e_{I \setminus i} \in \Mmac_{k-1}$, d'où le résultat puisque 
$\partial(\uP)(X^\alpha e_I)$ est une somme de $\pm P_i X^\alpha e_{I \setminus i}$ pour $i \in I$.
\end{proof}

\subsubsection{$\bullet$ L'induit-projeté $\beta_k(\protect\uP) : \Smac_k \rightarrow \Mmac_{k-1}$
  et l'isomorphisme monomial  $\varphi : \Mmac_{k-1} \rightarrow \Smac_k$ pour $k \geqslant 1$}

\begin{defn} \label{DefBetak}
On désigne par $\beta_k(\uP) : \Smac_k \rightarrow \Mmac_{k-1}$ 
l'application $\bfA$-linéaire 
qui est l'induit-projeté de $\partial_k(\uP)$, 
\idest{} $\beta_k = \pi_{\Mmac_{k-1}} \circ \partial_k(\uP) \circ \iota_{\Smac_k}$.
\end{defn}

\begin{prop}\label{varphiIso}
Pour le jeu étalon $\uX^D$, l'application $\beta_k(\uX^D) : \Smac_k \rightarrow \Mmac_{k-1}$ réalise
$X^\alpha e_I  \, \mapsto\,  X^\alpha X_i^{d_i} e_{I \setminus i}$ 
où $i = \min I$. 
C'est un isomorphisme de $\bfA$-modules gradué de degré $0$ dont l'inverse, noté $\varphi$,
est donné par 
$$
\boxed{
\varphi : 
\begin{array}[t]{rcl}
\Mmac_{k-1} & \longrightarrow & \Smac_k \\ [0.3em]
X^\beta e_J & \longmapsto & \dfrac{X^\beta}{X_i^{d_i}}\, e_i\wedge e_J \quad 
\text{où $i = \minDiv(X^\beta)$}
\end{array}
}
$$
\end{prop}

\label {NOTA14-betak}%
\label {NOTA14-varphi}%

\begin{proof}
Justifions tout d'abord l'expression de $\beta_k(\uX^D)$. 
En notant $\varepsilon_i(I)$ le nombre d'éléments de $I$ strictement inférieurs à $i$,
rappelons que
$$
\partial(\uX^D) : \quad 
X^\alpha e_I \ \longmapsto \ 
\sum_{i \in I} (-1)^{\varepsilon_i(I)} X^\alpha X_i^{d_i} e_{I \setminus i}
$$
Supposons $X^\alpha e_I \in \Smac_k$.
La projection sur $\Mmac_{k-1}$ de la somme ci-dessus passe par la détermination
des $i \in I$ tels que 
$X^\alpha X_i^{d_i} e_{I \setminus i} \in \Mmac_{k-1}$. 
En posant $X^\beta = X^\alpha X_i^{d_i}$, 
il s'agit de décider à quelle condition sur $i \in I$, le monôme
extérieur $X^\beta e_{I\setminus i}$ est dans $\Mmac_{k-1}$. 
On va voir que c'est le cas si et seulement si $i = \min I$.

On distingue deux cas.  Premier cas : $i = \min I$. Alors $X^\beta
e_{I\setminus i}$ est dans $\Mmac_{k-1}$ car on dispose de $i < \min(I
\setminus i)$ vérifiant $\beta_i \geqslant d_i$ (et ceci quel que soit
le statut de $X^\alpha e_I$, dans $\Smac_k$ ou pas !).  Deuxième cas :
$i > \min I$.  Comme $X^\alpha e_I$ est dans $\Smac_k$ (cette fois,
cette information est utilisée), pour tout $j < \min I$, on a
$\alpha_j < d_j$. Or, vu que $\min I < i$, on a $\min(I\setminus i)
= \min I$, donc pour tout $j < \min(I\setminus i)$, on a $\beta_j <
d_j$ ce qui traduit $X^\beta e_{I\setminus i} \in \Smac_{k-1}$.

Par conséquent :
$$
\beta_k(\uX^D) :
\begin{array}[t]{rcl}
\Smac_k & \longrightarrow & \Mmac_{k-1} \\ [0.3em]
X^\alpha e_I & \longmapsto & X^\alpha X_i^{d_i} e_{I \setminus i}
\quad \text{où $i = \min I$}
\end{array}
$$
Reste à vérifier que $\varphi \circ \beta_k(\uX^D) = \id_{\Smac_k}$ et
$\beta_k(\uX^D) \circ \varphi = \id_{\Mmac_{k-1}}$. 
Cela repose sur les deux implications :
$$
\minDiv(X^\alpha) \geqslant i \ \Rightarrow \ 
\minDiv\big(X^\alpha X_i^{d_i}\big) = i
\qquad \qquad \text{et} \qquad \qquad 
i < \min J \ \Rightarrow \ 
\min (i \vee J) = i
$$
\end{proof}

Pour résumer la situation, voici deux petits schémas
(la présence du $0$ dans la matrice de $\partial_k(\uP)$ tient compte de~\ref{CuriositeMcomplexe} 
mais cette information ne sera pas utilisée dans la suite) :

\begin{minipage}[t]{0.4\textwidth}
$$
\partial_k(\uP)  \ =\ 
\NorthEastBordermatrix{ 
\mbox{\footnotesize $\Mmac_k$}  & \mbox{\footnotesize $\Smac_k$} & \\
\star & \cercle{$\beta_k$} & \Heti{\Mmac_{k-1}} \\
0 & \star & \Heti{\Smac_{k-1}} \\
}
$$
\end{minipage} \qquad
\begin{minipage}[t]{0.4\textwidth}
$$
\xymatrix @M=0.2pc @R=0.2pc @C=5pc{
\Mmac_k & \Mmac_{k-1}  \ar@<0.5ex>@{-->}[ddl]^{\varphi} \\
\oplus & \oplus \\
\Smac_k \ar@<0.5ex>[ruu]^-{\beta_k(\uP)} & \Smac_{k-1} \\
}
$$
\end{minipage}

\subsubsection{$\bullet$ L'endomorphisme $B_k = B_k(\protect\uP)$ de $\Mmac_{k-1}$ pour $k \geqslant 1$}

Dans la définition qui vient, l'endomorphisme $B_k$ du $\bfA[\uX]$-module $\Mmac_{k-1}$
est un $\bfA$-endomorphisme  (car $\varphi$ est seulement $\bfA$-linéaire).

\begin{defn}\label{EndoBk}
L'endomorphisme $B_k(\uP)$ du $\bfA$-module $\Mmac_{k-1}$
est celui donné par la formule $B_k(\uP) = \beta_k(\uP) \circ \varphi \overset {\rm def.}{=}
\pi_{\Mmac_{k-1}} \circ \partial_k(\uP) \circ \varphi$.
\end{defn}

\label {NOTA14-Bk}%
\index{B@les endomorphismes!$B_k(\uP)$ du $\bfA$-module de Macaulay $\Mmac_{k-1}$}%

\smallskip

Insistons sur le décalage d'une unité $k \leftrightarrow k-1$:
$B_k(\uP) \overset{\rm def.}{=} \beta_k(\uP) \circ\beta_k(\uX^D)^{-1}$
opère sur $\Mmac_{k-1}$. Ceci est voulu de manière à ce que pour $k =
1$, $B_1(\uP)$ opère sur $\Mmac_0$, en cohérence avec $B_1(\uP) =
W_1(\uP)$ et $\Mmac_0 = \Jex_1$.  Voir également l'exemple suivant
où nous montrons que $B_n(\uP)$ s'identifie à $W_n(\uP)$.
Ceci nous a conduit à rejeter l'autre choix $\beta_k(\uX^D)^{-1} \circ \beta_k(\uP)$ opérant
sur~$\Smac_k$.

\begin {exemple} [$B_n(\uP)$ s'identifie à $W_n(\uP)$]
\label{BnEgalWn}

L'idéal monomial $\Jex_n$ et le sous-module de Macaulay $\Mmac_{n-1}$ de $\rmK_{n-1}$
sont deux $\bfA[X]$-modules libres de rang~1:
$$
\Jex_n = \bfA[\uX]\,X^D \overset{\rm def.}{=} \bfA[\uX]\,X_1^{d_1} \cdots X_n^{d_n},
\qquad\qquad
\Mmac_{n-1} = \bfA[\uX]\, X_1^{d_1}\, e_2\wedge\cdots\wedge e_n
$$ 
L'égalité de droite résulte du fait que $\Mmac_{n-1}$ est le $\bfA[X]$-module engendré par les $X_i^{d_i} e_I$
avec $\#I = n-1$ et $i < I$: la seule possibilité est $i = 1$ et $I = \{2,\dots,n\}$.

L'isomorphisme $\Mmac_{n-1} \to \Jex_n$ de $\bfA[\uX]$-modules qui transforme
$X_1^{d_1}\, e_2\wedge\cdots\wedge e_n$ en $X^D$ peut-être décrit comme la
restriction du pont homologique $\psi$ vers le degré $0$, ce dernier étant
$$
\psi : \rmK(X^D) \to \bfA[\uX], \qquad X^\beta e_J \to X^\beta X^{D(J)}
\qquad \text{où} \quad X^{D(J)} = \prod_{j \in J} X_j^{d_j}
$$
Cette application $\psi$ est un morphisme gradué de $\bfA[\uX]$-modules dont
nous reparlerons dans le chapitre~\ref{ChapBW}.  Montrons que sa restriction à
$\Mmac_{n-1}$, isomorphisme de $\Mmac_{n-1}$ sur $\Jex_n$,
rend commutatif le diagramme:
$$
\xymatrix {
\Mmac_{n-1}\ar[d]_{B_n(\uP)} \ar[r]^{\psi}_{\simeq} &\Jex_n \ar[d]^{W_n(\uP)}\\
\Mmac_{n-1} \ar[r]^{\psi}_{\simeq} & \Jex_n \\
}
$$  
Pour $y = X^\alpha\, e_2\wedge\cdots\wedge e_n \in \Mmac_{n-1}$ donc
$\alpha_1 \geqslant d_1$, il s'agit de voir (en omettant $\uP$ des notations) que:
$$
(\psi \circ B_n)(y) = (W_n \circ \psi)(y)
$$
A gauche, on utilise $B_n = \pi_{\Mmac_{n-1}} \circ \partial_n(\uP) \circ \varphi$ avec
$$
\varphi(y) = z := \frac{X^\alpha}{X_1^{d_1}}\,e_1\wedge e_2\wedge\cdots\wedge e_n,
\qquad\qquad
B_n(y) = \big(\pi_{\Mmac_{n-1}} \circ \partial_n(\uP)\big) (z) =
\frac{X^\alpha}{X_1^{d_1}}\,P_1 \wedge e_2 \wedge \cdots \wedge e_n
$$
D'où:
$$
(\psi \circ B_n)(y) = \frac{X^\alpha}{X_1^{d_1}}\,P_1 X_2^{d_2} \cdots X_n^{d_n}
$$
A droite, en  utilisant $\alpha_1 \geqslant d_1$:
$$
\psi(y) = X^\alpha\, X_2^{d_2} \cdots X_n^{d_n}
\qquad\qquad
(W_n \circ \psi)(y) = \frac{X^\alpha}{X_1^{d_1}}\,P_1 X_2^{d_2} \cdots X_n^{d_n}
$$
On a bien l'égalité $(\psi \circ B_n)(y) = (W_n \circ \psi)(y)$.

\smallskip
Puisque $B_n(\uP)$ et $W_n(\uP)$ sont conjugués par un isomorphisme gradué de degré $0$,
on en déduit que $B_{n,d}(\uP)$ et $W_{n,d}(\uP)$ sont conjugués, à fortiori
$\det B_{n,d}(\uP) = \det W_{n,d}(\uP)$ pour tout $d$.
\end {exemple}

\begin{prop}\label{ExpressionBkXbetaeJ}
L'image de $X^\beta e_J \in \Mmac_{k-1}$ par l'endomorphisme $B_k(\uP)$ est donnée par :
$$
\big(B_k(\uP)\big)(X^\beta e_J)
\ = \ 
\pi_{\Mmac_{k-1}} 
\Bigg( 
\dfrac{X^\beta}{X_i^{d_i}}\  P_i \ e_J 
\quad + \quad 
\sum_{j \in J} 
\pm\, \dfrac{X^\beta}{X_i^{d_i}} \ 
P_j \ e_{i \vee J \setminus j}
\Bigg)
\qquad 
\text{où $i = \minDiv(X^\beta)$}
$$
\end{prop}

\begin{proof}
Par définition de $X^\beta e_J \in \Mmac_{k-1}$, on a $i < \min J$. 
L'endomorphisme $B_k$ est égal à $\pi_{\Mmac_{k-1}} \circ \partial_k \circ \varphi$ où 
$\varphi : \Mmac_{k-1} \rightarrow \Smac_k$ est l'isomorphisme monomial de référence donné par
$\varphi(X^\beta e_J) = \frac{X^\beta}{X_i^{d_i}} e_i \wedge e_J$.
D'où le résultat par définition de la différentielle $\partial_k$ du complexe de Koszul de $\uP$
et du fait que $i < \min J$.
\end{proof}

\medskip

La proposition suivante, analogue à la proposition~\ref{ProprieteWh},
ne présente pas de difficulté.

\index{B@les endomorphismes!$B_{k,d}(\uP)$ de $\Mmac_{k-1,d}$}%

\begin{prop} \label{BkPropriete}
\leavevmode
\begin{enumerate}[\rm i)]
\item
Pour le jeu étalon $\uX^D$, on a $B_k(\uX^D) = \Id$ de sorte que $\det B_{k,d}(\uX^D) = 1$.

\item
L'endomorphisme $B_1(\uP)$ de $\Mmac_0$ coïncide avec l'endomorphisme $W_1(\uP)$ de $\Jex_1 = \Mmac_0$.

\item
Si $\uP$ est le jeu générique, le déterminant $\det B_{k,d}(\uP)$ est primitif par valeur, à
fortiori régulier.

\item 
Pour deux jeux $\uP$ et $\uQ$ de même format $D$, on a $B_k(\uP + \uQ) = B_k(\uP) + B_k(\uQ)$.

\item
En particulier, pour une indéterminée $t$, on~a $B_k(\uP + t\uX^D) = B_k(\uP) + t \,\Id$.
Chaque $\det B_{k,d}(\uP + t\uX^D)$ est un polynôme unitaire en $t$, de
degré égal à $\dim \Smac_{k,d}$, a fortiori un élément régulier de $\bfA[t]$.
\end{enumerate}
\end{prop}

\begin{rmq} \label{RmqConstructionWh}
Suite  à la remarque~\ref{FenetreObservationPourSyl},
l'endomorphisme $B_k(\uP)$ est défini à partir de l'application linéaire~$\beta_k(\uP)$
issue de $\partial_k(\uP)$ :
$$
\xymatrix @M=0.4pc @C=1.2cm{
\Smac_k \ar@{^{(}->}[r] \ar@/_0.5cm/[rrr]_-{\beta_k(\uP)}  & 
\rmK_k \ar[r]^-{\partial_k(\uP)} & 
\rmK_{k-1} \ar@{->>}[r] & 
\Mmac_{k-1} \ar@{-->}@<-1ex>@/_0.5cm/[lll]_-{\varphi} \\
}
\qquad 
\text{\ et }
\qquad 
B_k(\uP) = \beta_k(\uP) \circ \varphi
$$
Pour $k=1$, le schéma prend la forme suivante en remarquant que $\partial_1(\uP)$
est l'application de Sylvester $\Syl(\uP)$ et 
$\Mmac_0 = \Jex_1$ : 
$$
\xymatrix @M=0.4pc @C=1.2cm{
\varphi(\Jex_1) \ar@{^{(}->}[r] \ar@/_0.5cm/[rrr]_-{\beta_1(\uP)}  & 
\rmK_1 \ar[r]^-{\Syl(\uP)} & 
\rmK_{0} \ar@{->>}[r] & 
\Jex_1 \ar@{-->}@<-1ex>@/_0.5cm/[lll]_-{\varphi} \\
}
\qquad 
\text{\ et }
\qquad 
B_1(\uP) = \beta_1(\uP) \circ \varphi
$$
On peut généraliser cela à $\Jex_h$ 
pour $h \geqslant 1$ avec le schéma ci-dessous.
Naît alors l'application $\bfA$-linéaire 
$\scrW_h(\uP) := \pi_{\Jex_h} \circ \Syl(\uP)_{|\varphi(\Jex_h)}$ 
qui vérifie en particulier $\scrW_1(\uP) = \beta_1(\uP)$.
Et finalement, on récupère l'endomorphisme~$W_h(\uP)$ par la formule $\scrW_h(\uP) \circ \varphi$.
$$
\xymatrix @M=0.4pc @C=1.2cm{
\varphi(\Jex_h) \ar@{^{(}->}[r] \ar@/_0.5cm/[rrr]_-{\scrW_h(\uP)}  & 
\rmK_1 \ar[r]^-{\Syl(\uP)} & 
\rmK_{0} \ar@{->>}[r] & 
\Jex_h \ar@{-->}@<-1ex>@/_0.5cm/[lll]_-{\varphi} \\
}
\qquad 
\text{\ et }
\qquad 
W_h(\uP) = \scrW_h(\uP) \circ \varphi
$$
\end{rmq}

\subsection{Les scalaires $\Delta_{k,d}(\protect\uP)$ et la forme déterminant de Cayley du
            complexe $\rmK_{\sbullet,d}(\protect\uP)$}
\label{SectionDeltakdMacaulay}

Lorsque $\uP$ est régulière, le complexe $\rmK_{\sbullet,d}(\uP)$ est
un $\bfA$-complexe de modules libres \textit{exact} dont nous avons
noté $\chi_d$ la caractéristique d'Euler-Poincaré. Etant a fortiori de
Cayley, on peut donc lui appliquer les résultats de la
section~\ref{SousSectionLesDeltak}.

L'introduction des $B_{k,d}(\uP)$
et des $\Delta_{k,d}(\uP)$ va fournir un éclairage sur ``la'' forme
générateur de son invariant de Cayley
$\CayleyVect(\rmK_{\sbullet,d})$, sous-module de
$\bigwedge^{\chi_d}(\bfA[\uX]_d)^\star$, défini
en~\ref{CayleyDetVectoriel}. Ceci va  permettre en particulier, pour $d
\ge \delta+1$, de relier $\calR_d(\uP)$ et $\Delta_{1,d}(\uP)$,
conduisant ainsi, pour le résultant $\Res(\uP)$, à une expression de
nature différente de la formule de Macaulay figurant
en~\ref{MacaulayParRecurrence}, et permettant plus tard de retrouver cette dernière d'une manière
que nous estimons plus naturelle.
Idem en degré $\delta$: nous retrouverons ici $\omegares$ sous un autre habillage.

Cependant le cas du degré $\ge \delta$ n'est pas l'unique but visé.
En plus de la mise en place de nouvelles notions, il s'agit d'apporter
un autre point de vue sur certains chapitres précédents (par exemple
sur la preuve de la formule de Macaulay) et de préparer, en
collaboration avec le chapitre suivant, l'établissement en degré $d$
quelconque de résultats non abordés jusqu'à maintenant.

\medskip

Comme nous comptons à la fois compléter certains résultats et en
revisiter d'autres de manière autonome (indépendamment de preuves
fournies auparavant), il nous semble opportun de faire le point de \og
là où nous en sommes, à ce moment de notre étude \fg{} en commençant par rappeler
quelques notations. Deux dimensions sont attachées à la décomposition
$\bfA[\uX]_d = \Jex_{1,d} \oplus \Smac_{0,d}$, décomposition qui ne
dépend que du format $D$:
$$
s_d = r_{1,d} = \dim \Jex_{1,d}, \qquad \chi_d = r_{0,d} = \dim\Smac_{0,d}
$$
En fonction de $\uP$, l'entier $s_d$ est aussi le rang
attendu de l'application de Sylvester $\Syl_d(\uP)$ et $\chi_d$ la
caractéristique d'Euler-Poincaré du complexe $\rmK_{\sbullet,d}(\uP)$,
cf.~\ref{chiProprieteComplements}.
Ces deux dimensions sont évidemment reliées par $s_d + \chi_d= \dim \bfA[\uX]_d$.

\subsubsection*{``La'' forme $\mu_d : \bigwedge\nolimits^{\chi_d}(\bfA[\uX]_d) \to \bfA$
du complexe $\rmK_{\sbullet,d}(\uP)$  (épisode I, retour sur le passé)}

Reprenons l'énoncé \og \textit{Module librement résoluble versus
module de MacRae} \fg{} figurant en~\ref{LibrementResolubleImpliqueMacRae} :

\begin{quote}
{\it
Un module $M$ librement résoluble est un module de MacRae.

\smallskip
Précisément, soit $M$ résolu par un complexe libre
$(F_\sbullet, u_\sbullet)$ de caractéristique d'Euler-Poincaré~$c$.  Alors $M$ est
un module de MacRae de rang $c$ dont la forme invariant de MacRae \mbox{$\vartheta
: \BW^c(M) \to \bfA$} s'obtient par passage au quotient du déterminant de
Cayley $\mu : \bigwedge^{c}(F_0) \to \bfA$ de~$F_\sbullet$.

\smallskip
Avec un léger abus de langage, le déterminant de Cayley de
$F_\sbullet$ et l'invariant de MacRae de~$M$ coïncident:
$$
\CayleyVect(F_\sbullet) \ \buildrel {\rm abus} \over = \ 
\MacRaeVect(M) 
$$
\`A gauche, il s'agit d'un sous-module de $\bigwedge^{c}(F_0)^\star$ et à droite
d'un sous-module de~$\bigwedge^{c}(M)^\star$.
}
\end{quote}

\medskip

Insistons sur le fait qu'ici $d$ est quelconque.  Lorsque $\uP$ est
régulière, le résultat ci-dessus s'applique au $\bfA$-module $\bfB_d$,
librement résolu par le complexe $\rmK_{\sbullet,d}(\uP)$, et donc, de ce fait,
de MacRae de rang $\chi_d$. On dispose ainsi d'une forme $\mu_d$,
définie à un inversible près, et d'un idéal de type fini $\fb_d$
vérifiant :
$$
\DVect_{s_d}(\Syl_d) = \mu_d\,\fb_d  \quad
\text{où $\mu_d : \bigwedge\nolimits^{\chi_d}(\bfA[\uX]_d)\to\bfA$
est sans torsion et $\Gr(\fb_d) \geqslant 2$}
$$
Nous allons mettre en veilleuse $\bfB_d$ sans oublier pour autant que
$\mu_d$ passe au quotient. Nous resterons principalement au niveau
$\bfA[\uX]_d$ et nous n'utiliserons pratiquement pas la forme $\overline{\mu_d}$,
pgcd-fort du sous-module déterminantal~$\FittVect_{\chi_d}(\bfB_d)$
de $\BW^{\chi_d}(\bfB_d)^\star$. Toutefois, il nous
arrivera d'utiliser indifféremment la notation MacRae $\MacRaeVect(\bfB_d) = \bfA
\overline{\mu_d}$ au lieu de celle de l'invariant de Cayley
$\CayleyVect(\rmK_{\sbullet,d}) = \bfA\mu_d$.

Non seulement nous resterons au niveau $\bfA[\uX]_d$ mais nous verrons
dans la section~\ref{CanonicalCayleyDetSection} qu'il est préférable
de masquer l'algèbre extérieure en considérant $\mu_d$ comme une forme
$\chi_d$-linéaire alternée $\bfA[\uX]_d^{\chi_d} \to \bfA$.
En remplaçant le domaine de définition $\bfA[\uX]_d^{\chi_d}$ de cette forme par un
module isomorphe adéquat, on obtiendra un objet intrinsèque de nom $\Det_d$.

\medskip

Plusieurs questions se posent.

\begin {enumerate}
\item
Parmi toutes ces formes $\mu_d$, peut-on en sélectionner une particulière 
présentant un caractère intrinsèque?  

\item
Si oui, peut-on en fournir une expression algébrique?

\item
Si oui, dans quelle mesure cette expression algébrique est-elle susceptible
de fournir un moyen de calcul? Ce moyen de calcul nécessite-t-il des hypothèses
supplémentaires sur $\uP$? (nous pensons à la présence de dénominateurs).
\end {enumerate}  

\bigskip

$\rhd$
Nous pouvons répondre \emph {partiellement} à la question 1.
Pour $d\geqslant \delta$, en~\ref{PoidsNormalisationMacRae}, en passant en générique,
en utilisant le poids en $P_i$ et le jeu étalon généralisé, nous
avons produit un générateur privilégié $\mu_d$ (normalisé vis-à-vis du jeu étalon)
de l'invariant de MacRae $\MacRaeVect(\bfB_d)$.

\smallskip

En degré $\delta$ pour lequel $\chi_\delta=1$, la forme $\mu_\delta$ distinguée 
a pris le nom de $\omega_\res \in
\bigwedge^{1}(\bfA[\uX]_\delta)^\star = \bfA[\uX]_\delta^\star$, forme linéaire
normalisée pour le jeu étalon (confer~\ref{PoidsNormalisationMacRae}
pour les détails) : c'est le pgcd fort des formes linéaires
déterminantales de $\Syl_\delta$.

En degré $\geqslant \delta+1$, on a $\chi_d = 0$. Ainsi 
$\bigwedge\nolimits^{\chi_d}(\bfA[\uX]_d)^\star$ s'identifie canoniquement à $\bfA^\star = \bfA$
et la forme linéaire $\mu_d$ bien particulière dont il est question est 
\textit{ici} un scalaire $\calR_d \in \bfA$ normalisé, pgcd fort 
des mineurs pleins de $\Syl_d$.

\smallskip

Il nous semble important de préciser que cette normalisation a pu être
réalisée uniquement à partir de la \textit{présentation canonique} du
module $\bfB_d$ par l'application de Sylvester $\Syl_d(\uP)$, sans
utiliser le complexe~$\rmK_{\sbullet,d}(\uP)$ dans sa globalité.
Cependant, lorque $\uP$ est régulière, l'existence même de l'invariant
de MacRae $\MacRaeVect(\bfB_d)$ repose sur la structure multiplicative
du complexe $\rmK_{\sbullet,d}(\uP)$.


\bigskip

$\rhd$
Qu'en est-il de la question~2 (en degré $\ge \delta$)?  On dispose des
formules de Macaulay figurant en~\ref{MacaulayParRecurrence}
et~\ref{MacaulayParRecurrenceDegreeDelta}, formules simples et
efficaces (en degré $\ge \delta+1$, il s'agit de $\calR_d = \det W_{1,d}/\det
W_{2,d}$).

\smallskip

Nous estimons cependant que la mise en place de ces formules n'est pas
du tout satisfaisante.  Un mot concernant la \emph {nature} des
théorèmes~\ref{MacaulayParRecurrence} et
\ref{MacaulayParRecurrenceDegreeDelta}.  La preuve a été établie par
récurrence sur~$n$ (comment a-t-on pu deviner ces formules?), conférant à ce
résultat un certain mystère et une grande insatisfaction.  Il semble
que nous ne soyons pas seuls à le penser.

Par exemple, dans \cite[haut de la page 8 de l'introduction]{GKZ}, les auteurs
I.M. Gelfand, M.M. Kapranov \& A.V. Zelevinski écrivent en 1994:

\begin {quote}
... Macaulay made another intriguing contribution to the theory
by giving an ingenious refinement of the Cayley method. It would be interesting
to put this approach in the general framework of this book.
\end {quote}

Il en est de même plus récemment (2021) de C. D'Andrea et ses co-auteurs dans
\cite[premières lignes de l'introduction]{AJS}:

\begin {quote}
In [Mac1902], Macaulay introduced the notion of homogeneous resultant,
extending the Sylvester resultant to systems of homogeneous
polynomials in several variables with given degrees. In the same
paper, he also presented an intriguing family of formulae, each of
them allowing to compute it as the quotient of the determinant of a
Sylvester type square matrix by one of its principal minors.
\end {quote}

En plus de la récurrence sur $n$, la preuve nécessite d'avoir avec soi
l'égalité des invariants de MacRae des~$(\bfB'_d)_{d \geqslant
  \delta}$, cf.~\ref{MacRaeEqualities}, ce qui n'est pas une évidence!

\subsubsection*{Un aperçu de ce que nous allons réaliser pour
  répondre de manière positive aux 3 questions}

Nous proposons dans ce chapitre et le suivant une approche
radicalement différente de celle par récurrence sur $n$.  Pour nous
simplifier la vie, dans cette description, nous supposons~$\uP$
générique et nous ne le faisons pas intervenir dans les notations.

\smallskip

Nous allons principalement travailler à un degré $d$ fixé
en nous aidant de la décomposition de Macaulay
$(\Mmac_{\sbullet,d}, \Smac_{\sbullet,d}, \varphi)$ de
$\rmK_{\sbullet, d}$.  Tout choix d'un système d'orientations de cette
décomposition (cf. la définition en
section~\ref{SousSectionLesDeltak}) fournit une forme $\mu_d \in
\bigwedge\nolimits^{\chi_d}(\bfA[\uX]_d)^\star$ bien précise,
générateur de $\CayleyVect(\rmK_{\sbullet,d})$, et \emph
{simultanément} une expression de cette forme utilisant cette fois-ci
\textit{toutes les différentielles} du complexe.
Dans la suite, nous ne considérerons que des systèmes monomiaux
d'orientations: cela signifie que l'orientation d'un sous-module monomial est le
produit extérieur de sa base monomiale ordonnée arbitrairement.

\smallskip

Attention ici à la terminologie et à l'histoire.  On attribue à
Macaulay deux notions bien distinctes : d'une part les formules
impliquant~$W_{1,d}$ et $W_{2,d}$ relatives à la \textit{première}
différentielle et d'autre part, la décomposition de
$\rmK_{\sbullet,d}$ qui porte son nom et donne naissance au mécanisme
de Cayley, mécanisme qui fait apparaître les formules de quotients
alternés relatives à \textit{toutes les} différentielles. De quoi s'y
perdre.

\medskip

Pour l'instant, nous n'avons pas encore répondu complètement à la question~1.  Car
au départ, dans notre description, il y a le \emph{choix} d'un système
monomial d'orientations qui conditionne la forme $\mu_d$ et fait que
celle-ci est définie au signe près.  Et un choix, c'est là le hic,
ce n'est pas intrinsèque. Mais alors que tout semble englué dans des histoires
d'orientations, nous verrons que

\smallskip
\centerline{\fbox{$
\begin {array}{c}    
\text{la base monomiale $\calS_{0,d}$ de $\Smac_{0,d}$, de cardinal $\chi_d$, permet d'élaborer} \\
\text{une forme $\chi_d$-linéaire alternée sur $\bfB_d$, qui est intrinsèque} \\
\end {array}
$}}
\smallskip

Voir à ce propos le
lemme~\ref{MultiplicateursOrientationsDecomposition} qui affirme que
l'orientation sur $\Smac_{0,d}$ rigidifie $\mu_d$.

\medskip

Peut-on orienter canoniquement $\Smac_{0,d}$? Oui pardi si $\dim
\Smac_{0,d} \le 1$! C'est le cas pour $d \ge \delta$: pour $d=\delta$,
on a $\Smac_{0,d} = \bfA X^\emouton$, orienté canoniquement
par $X^\emouton$ (il faudrait être tordu pour orienter~$\Smac_{0,d}$
par~$-X^\emouton$) tandis que $\Smac_{0,d} = 0$ pour $d \ge \delta+1$.
Le lecteur vérifiera que pour le format $D =
(1,1,\dots,e)$, on a $\dim \Smac_{0,d} \le 1$ pour tout~$d$, plus
précisément, $\Smac_{0,d} = \bfA X_n^d$ pour $d \le \delta$ et
$\Smac_{0,d}=0$ pour $d \ge \delta + 1$. Ceci a comme conséquence,
pour ce format,  qu'il y a, pour $d \le \delta$, 
une forme linéaire intrinsèque $\mu_d : \bfA[\uX]_d \to \bfA$, qui
passe au quotient sur $\bfB_d$.

Mais en général, pour $d <\delta$, on a $\dim \Smac_{0,d} \ge 2$ et
aucun moyen d'orienter canoniquement $\Smac_{0,d}$.  Cependant, nous
ferons ce qu'il faut plus tard dans la
section~\ref{CanonicalCayleyDetSection} pour obtenir une forme
$\chi_d$-linéaire alternée intrinsèque sur $\bfA[\uX]_d$ (qui
passe au quotient sur $\bfB_d$).

\medskip

Lorsque $d \ge \delta$, nous allons montrer dans cette section (cf les propositions \ref{mu_d=R_d}
et \ref{mu_delta=omegares}) que
la forme $\mu_d$ définie ici par l'orientation canonique (et triviale) de
$\Smac_{0,d}$ coïncide avec la forme $\mu_d$ spécifiée
en~\ref{PoidsNormalisationMacRae}.  Nous obtiendrons une expression
\emph{différente} de celles fournies en~\ref{MacaulayParRecurrence}
et~\ref{MacaulayParRecurrenceDegreeDelta}; mais que le lecteur soit
rassuré: nous retrouverons plus tard (cf. les théorèmes
\ref{DetBkdProdBinomialDetWhd} et \ref{DeltakdProdBinomialDetWhd})
les formules de Macaulay d'une
toute autre manière.  Et nous aurons ainsi court-circuité complètement
la section~\ref{sousSectionMacaulayRecurrence}.

\medskip

Signalons tout de même deux problèmes majeurs. Il y a celui de la
comparaison entre le degré~$d$ et~$d+1$ d'une part et d'autre part le
fait qu'en degré $\delta$, nous n'avons pas réussi à faire cohabiter
le bezoutien~$\nabla$ de $\uP$ et la décomposition de Macaulay de
$\rmK_{\sbullet,\delta}(\uP)$.  Nous parviendrons tout de même à
montrer (cf. en fin de section \ref{sectionProfondeurBk}) que
$\Delta_{1,d} = \Delta_{1,d+1}$ pour $d \ge \delta+1$ retrouvant ainsi
l'égalité des invariants de MacRae des $(\bfB_d)_{d\ge \delta+1}$.
Mais nous n'avons rien de nouveau concernant le lien entre le
degré~$\delta$ et~$\delta+1$.

\smallskip

Ceci termine l'aperçu. Place à la technique.

\subsubsection*{La structure multiplicative de $\rmK_{\sbullet,d}(\uP)$ enrichie
par la décomposition de Macaulay}

Nous allons appliquer une partie des résultats de la
section~\ref{SousSectionLesDeltak} au complexe
$\rmK_{\sbullet,d}(\uP)$ dont nous assurons l'exactitude en supposant
$\uP$ régulière.  Nous reviendrons dans le paragraphe suivant sur le statut de $\uP$.
Ci-dessous, la plupart du temps, comme d'habitude, nous omettons $\uP$ des notations.

Fixons un système monomial $\bfe$ d'orientations de $\rmK_{\sbullet,d}(\uX^D)$ (et les
habituels isomorphismes déterminantaux
associés notés~$\sharp$). D'après~\ref{Factorisation}, ce système 
définit \textit{de manière unique} une suite de vecteurs $\Theta_{k,d}
\in \bigwedge^{r_{k-1,d}}(\rmK_{k-1,d})$ pour $1 \leqslant k
\leqslant n+1$ vérifiant $\Theta_{n+1,d} = 1$ et la factorisation:
$$
\bigwedge\nolimits^{r_{k,d}}(\partial_{k,d}) \ = \ 
\Theta_{k,d} \, \Theta_{k+1,d}^\sharp
\qquad \text{où} \qquad
\Theta_{k+1,d}^\sharp = \oriented{\Theta_{k+1,d} \, \wedge \, \sbullet\,}_{\bfe_{\rmK_{k,d}}}
\leqno\Bigl(
\begin{array}{c}
  \text{matricielle} \\ \text{colonne} \times \text{ligne} \\
\end{array}
\Bigr)  
$$
Puisque $\calD_{r_{k,d}}(\partial_{k,d}) = \rmc(\Theta_{k,d})\rmc(\Theta_{k+1,d}) \subset \rmc(\Theta_{k,d})$,
on dispose de l'inégalité de profondeur $\Gr(\Theta_{k,d}) \geqslant k$.

\smallskip

Tenons maintenant compte de la décomposition de Macaulauy $(\Mmac_{\sbullet,d}, \Smac_{\sbullet,d}, \varphi)$
et considérons un système monomial d'orientations de cette décomposition adapté au système~$\bfe$.
L'évaluation en $\bfe_{\Smac_{k,d}}$ de l'égalité matricielle ci-dessus fournit (point ii) de
\ref{SuiteDeltak} appliqué à $\Delta_{k+1,d}$), l'égalité
$$
\bigwedge\nolimits^{r_{k,d}}(\partial_{k,d})(\bfe_{\Smac_{k,d}})
\ =\ 
\Theta_{k,d} \, \Delta_{k+1,d}
\leqno (\text{vectorielle})
$$
En projetant sur $\bfe_{\Mmac_{k-1,d}}$, on obtient alors (cette fois par définition
de $\Delta_{k,d}$, cf. \ref{SuiteDeltak}), l'égalité
$$
\det B_{k,d}\ = \ \Delta_{k,d}\, \Delta_{k+1,d}
\leqno (\text{scalaire})
$$
dans laquelle les scalaires ne dépendent plus d'aucun système d'orientations,
seulement de la décomposition de Macaulay choisie (pour l'instant, nous ne
connaissons qu'une telle décomposition mais il y en a d'autres).

\label {NOTA14-DeltakdP}%

\medskip

De l'égalité $(\text{vectorielle})$, on tire le résultat de divisibilité ci-dessous.

\begin{prop} \label{DeltakdDiviseMineurs}
Soit $\uP$ une suite régulière.

Le scalaire $\Delta_{k+1,d}(\uP)$ divise tous
les mineurs d'ordre $r_{k,d}$ de $\partial_{k,d}(\uP)$ indexés par les
colonnes de~$\Smac_{k,d}$.

En particulier, $\Delta_{2,d}(\uP)$ divise tous les mineurs d'ordre $s_d$ de 
$\Syl_d(\uP)$ indexés par les colonnes de $\Smac_{1,d}$.
\end{prop}

\subsubsection*{Précision sur les $\Delta_{k,d}(\uP)$ selon le statut de $\uP$}

Aucune contrainte sur $\uP$ n'est nécessaire pour définir le déterminant $\det B_{k,d}(\uP)$.
En ce qui concerne les $\Delta_{k,d}(\uP)$, c'est un peu plus délicat:
si $\uP$ est régulière, le complexe $\rmK_{\sbullet,d}(\uP)$ est exact,
en particulier de Cayley, et les $\Delta_{k,d}(\uP)$ sont bien définis, 
mais \textit{non nécessairement réguliers}, pire, ils peuvent être nuls.
Ils vérifient en particulier:
$$
\det B_{k,d}(\uP) = \Delta_{k,d}(\uP)\, \Delta_{k+1,d}(\uP) 
\qquad \text{ et } \qquad 
\Delta_{n+1,d}(\uP) = 1
$$
Pour le jeu étalon, on a en revanche de la vraie régularité, puisque $\det B_{k,d}(\uX^D) = 1$, 
donc $\Delta_{k,d}(\uX^D) = 1$.

Comprenons que, même si $\uP$ est régulière, la relation ci-dessus ne permet pas toujours la détermination
(par récurrence) des $\Delta_{k,d}$ à partir des $\det B_{k,d}$ car certains de ces déterminants
peuvent être nuls. Il n'est donc \textit{pas} question d'écrire brutalement:
$$
\Delta_{k,d} \ = \ 
\dfrac{\det B_{k,d} \,\det B_{k+2, d} \, \cdots}
{\det B_{k+1, d} \, \det B_{k+3, d} \, \cdots}
$$
Une solution consiste à considérer la suite générique pour obtenir des \textit{identités algébriques} (donc valides 
pour n'importe quelle suite~$\uP$).
Dans un tel contexte générique, 
ces identités sont du type $bq = a$ où $a$ et $b$ sont deux intervenants \og connus\fg{} 
et réguliers de l'anneau $\bfA = \bfk[\indetsPi]$.
Cette dernière égalité dit que $b$ divise $a$ et elle \textit{fait naître} un nouvel 
intervenant : le quotient $q$.
Une fois ces égalités et nouveaux intervenants dégagés dans le cas générique, 
rien n'empêche d'opérer à une spécialisation et de définir les objets 
pour n'importe quelle suite $\uP$. Ainsi : 
\begin{quote}
$\Delta_{k,d}(\uP)$ est bien défini
pour n'importe quel système $\uP$, comme spécialisé du cas générique.
\end{quote}

\subsubsection*{Un moyen de forcer la régularité des $\det B_{k,d}$ et des $\Delta_{k,d}$}

Pour déterminer (et exprimer) les $\Delta_{k,d}(\uP)$ à partir des $\det B_{k,d}(\uP)$ 
par l'intermédiaire de la relation $\det B_{k,d}(\uP) = \Delta_{k,d}(\uP)\,\Delta_{k+1,d}(\uP)$, 
on peut procéder à une petite générisation de la suite~$\uP$, forçant les scalaires qui
interviennent à être réguliers.  Cette petite générisation signalée
par Demazure (cf.~\cite{Demazure1}) consiste à considérer une seule
indéterminée $t$ et à ajouter, à la suite quelconque $\uP$, le jeu $t
\uX^D$, multiple du jeu étalon ; autrement dit, il s'agit de
considérer la suite $(P_1 + tX_1^{d_1}, \dots, P_n + tX_n^{d_n})$ de $\bfA[t][\uX]$. 
D'après~\ref{BkPropriete}, chaque $\det B_{k,d}(\uP + t\uX^D)$ est un polynôme
unitaire en~$t$, donc régulier. Par conséquent, on montre par récurrence descendante 
que $\Delta_{k,d}(\uP + t\uX^D)$ est également régulier (on a $\Delta_{n+1,d} = 1$).
On revient au système $\uP$ par la spécialisation $t := 0$.

\subsubsection*{Une interprétation structurelle des $\Delta_{k,d}(\uP)$}

Ici on suppose la suite $\uP$ régulière.
Dans la suite, comme d'habitude, nous nous permettrons parfois d'omettre $\uP$ dans les notations.

\medskip
$\rhd$
\label{Delta1dMacRaeBd}
Commençons par le cas simple où $k = 1$ et $d \ge \delta+1$.
Après avoir fixé un système monomial d'orientations $\bfe$ de la
décomposition de Cayley de $\rmK_{\sbullet,d}$, on récupère
un vecteur $\Theta_{1,d} \in \BW^{\chi_d}(\bfA[\uX]_d)$ sans torsion.
Mais ici $\chi_d = 0$ donc $\BW^{\chi_d}(\bfA[\uX]_d) \simeq \bfA$,
de sorte que ce vecteur s'identifie à sa coordonnée $\Delta_{1,d}$ 
sur $\bfe_{\Mmac_{0,d}}$. Bilan: le scalaire (régulier) $\Delta_{1,d}$
est \emph{un} générateur de MacRae de $\bfB_d$.

\medskip
$\rhd$
Passons à $k$ quelconque en supposant  le scalaire $\Delta_{k,d}(\uP)$ régulier,
ce qui est le cas si $\uP$ est générique. 
La proposition~\ref{InterpretationStructurelleDeltak} fournit une caractérisation
structurelle du scalaire $\Delta_{k,d}$ : c'est un générateur de l'idéal 
déterminant de Cayley du complexe $\rmK_{\sbullet,d}$
tronqué en degré homologique $k-1$ et projeté au bout sur $\Mmac_{k-1,d}$.
$$
0 \to
\xymatrix @M=0.4pc @C=2cm{
\rmK_{n,d} \ar[r]^-{\partial_{n,d}} &
\rmK_{n-1,d} \ar[r]^-{\partial_{n-1,d}} &
\quad \cdots \quad \ar[r] &
\rmK_{k,d} \ar[r]^-{\pi_{\Mmac_{k-1,d}} \circ \partial_{k,d}} &
\Mmac_{k-1,d}
}
$$
Ce complexe est bien de caractéristique d'Euler-Poincaré nulle, il est de Cayley
parce que l'on a supposé~$\Delta_{k,d}(\uP)$ régulier, mais il n'y a aucune
raison qu'il soit exact en $\rmK_{k,d}$.

En conséquence, pour tout entier $d$, le scalaire $\Delta_{k,d}$,
lorsqu'il est régulier, engendre l'idéal invariant de MacRae du module
$\Coker (\pi_{\Mmac_{k-1,d}} \circ \partial_{k,d}) = \Mmac_{k-1,d}
/ \Im (\pi_{\Mmac_{k-1,d}} \circ \partial_{k,d})$, module de MacRae de
rang~0.

\medskip

Particularisons ce qui précède à $k = 1$ et $d\geqslant \delta+1$:
dans ce cas, le $\bfA$-module $\Mmac_{0,d}$ est
exactement~$\bfA[\uX]_d$.  On retrouve le fait que $\Delta_{1,d}$ est
un générateur de MacRae de $\bfB_d = \bfA[\uX]_d / \Im (\Syl_d)$, cas
particulier par lequel on avait commencé \emph {sans supposer
  $\Delta_{1,d}$ régulier}, uniquement sous le couvert de $\uP$
régulière.

\medskip

Revenons à $d$ quelconque.
Comme $\rmK_{k-1,d} = \Mmac_{k-1,d} \oplus \Smac_{k-1,d}$, on peut
remplacer le module $\Mmac_{k-1,d}$ par $\rmK_{k-1,d}/\Smac_{k-1,d}$:
$$
0 \to
\xymatrix @M=0.4pc @C=1.5cm{
\rmK_{n,d} \ar[r]^-{\partial_{n,d}} &
\rmK_{n-1,d} \ar[r]^-{\partial_{n-1,d}} &
\quad \cdots \quad \ar[r] &
\rmK_{k,d} \ar[r]^-{\overline{\partial_{k,d}}} &
\rmK_{k-1,d}/\Smac_{k-1,d}
}
$$
Pour $k=1$ et $d \in \bbN$, le conoyau de $\overline{\partial_{1,d}} = \overline{\Syl_d}$ est
isomorphe à $\bfB_d/\Smac_{0,d}$. Donc, lorsque $\Delta_{1,d}$ est
régulier, le $\bfA$-module $\bfB_d/\Smac_{0,d}$ est de MacRae de
rang $0$, d'invariant de MacRae engendré par $\Delta_{1,d}$.

\medskip
$\rhd$
Une mise en garde en ce qui concerne les $\Delta_{1,d}(\uP)$ pour $d \le \delta$,
en continuant à supposer $\uP$ régulière. On a toujours $\Delta_{1,0}(\uP) =1$.
Prenons trois indéterminées $a_1,b_3,c_2$ sur $\bbZ$ et $\uQ$ défini par:
$$
\uQ = (a_1X_1,\ b_3X_3,\ c_2X_2^e)
$$
C'est un petit exemple d'une suite $\uQ$ régulière, de format $D =
(1,1,e)$, donc de degré critique $\delta=e-1$ pour lequel nous allons
démontrer:
$$
\Delta_{1,d}(\uQ) = \begin {cases}
1 & \text{si } d=0 \\
0        & \text{si } 0 \le d \le \delta \\
c_2(-a_1b_3)^e & \text{si } d \ge \delta+1  \\
\end {cases}
$$
Pour les affirmations qui suivent dans un cadre plus général, les
détails sont parfois laissés à la charge du lecteur.  Tout d'abord
pour un format $D = (d_1,\dots, d_n)$ quelconque, on a $\rmK_{n,d} =
0$ pour $d < \sum_i d_i = \delta+n$. A fortiori $\Smac_{n,d} = 0$ donc
$\Mmac_{n-1,d}$, qui lui est isomorphe, est également nul. De sorte
que $\rmK_{n-1,d} = \Smac_{n-1,d}$ et $\Delta_{n-1,d}(\uP) = \det
B_{n-1,d}(\uP)$.  Pour $n=3$, cela s'énonce ainsi: lorsque $d \le
\delta+2$, on a $\rmK_{2,d} = \Smac_{2,d}$ et $\Delta_{2,d}(\uP) =
\det B_{2,d}(\uP)$.

Nous nous limitons désormais au format $(1,1,e)$ et à $d < \delta+2$.
Utilisons la formule
$$
\dim \rmK_{2,d} = \sum_{i < j} \binom{n + d - (d_i + d_j) - 1}{d - (d_i + d_j)}
\xymatrix @C = 2cm {\ar[r]^{D = (1,1,e)}_{d < \delta+2} & }
\dim \rmK_{2,d} = \binom{d}{d-2} +0 + 0 = \binom{d}{2} 
$$
La base monomiale de $\rmK_{2,d} = \Smac_{2,d}$ est formée des
$X^\beta e_{12}$ avec $|\beta| = d-2$ tandis que celle de
$\Mmac_{1,d}$ est constituée des $X_1X^\beta e_2$ avec $|\beta| =
d-2$. Il vient
$$
\varphi(X_1X^\beta e_2) = X^\beta e_{12}, \qquad\qquad
\partial_{2,d}(\uP)\big(\varphi(X_1X^\beta e_2)\big) = X^\beta (P_1e_2-P_2e_1)
$$
On en déduit que $B_{2,d}(\uP)$ ne dépend que de $P_1 = a_1X_1 +
a_2 X_2 + a_3 X_3$. En ordonnant la base $(X_1X^\beta
e_2)_{|\beta|=d-2}$ de $\Mmac_{1,d}$ selon l'ordre lexicographique sur
les $X_1X^\beta$ pour lequel $X_1 > X_2 > X_3$, on obtient que
$B_{2,d}(\uP)$ est triangulaire avec des $a_1$ sur la diagonale.
Par exemple, pour $d=4$:
$$
B_{2,d}(\uP) = 
\EastBordermatrix{
a_{1} & . & . & . & . & . & \Heti{X_{1}^{3}\,e_{2}} \\ 
a_{2} & a_{1} & . & . & . & . & \Heti{X_{1}^{2}X_{2}\,e_{2}} \\ 
a_{3} & . & a_{1} & . & . & . & \Heti{X_{1}^{2}X_{3}\,e_{2}} \\ 
. & a_{2} & . & a_{1} & . & . & \Heti{X_{1}X_{2}^{2}\,e_{2}} \\ 
. & a_{3} & a_{2} & . & a_{1} & . & \Heti{X_{1}X_{2}X_{3}\,e_{2}} \\ 
. & . & a_{3} & . & . & a_{1} & \Heti{X_{1}X_{3}^{2}\,e_{2}} \\ 
}
$$
Bilan:
$$
\Delta_{2,d}(\uP) = \det B_{2,d}(\uP) = a_1^{\binom{d}{2}}
$$
\c Ca, c'est du solide sur lequel on peut s'appuyer pour calculer $\Delta_{1,d}(\uP)$ via
$$
\det B_{1,d} = \Delta_{1,d} \Delta_{2,d}
$$
L'exposant de $a_1$ dans $\Delta_{2,d}(\uP)$ importe peu: ce qui compte, c'est
que l'on puisse contrôler la régularité de $\Delta_{2,d}(\uP)$.

Traitons enfin le petit système monomial $\uQ$. Pour $1 \le d \le
\delta$, nous affirmons que la ligne $X_2^d$ de $B_{1,d}(\uQ)$ est
nulle. En effet, la base monomiale de $\Smac_{1,d}$ est constituée de
monômes extérieurs $X^\gamma e_i$ avec $i = 1,2$ pour lesquels on a
$\coeff_{X_2^d}(X^\gamma Q_i) = 0$ vu le choix de $Q_1, Q_2$.  Bilan:
cette ligne nulle fait que $\det B_{1,d}(\uQ) = 0$ donc
$\Delta_{1,d}(\uQ) = 0$ dans la fourchette $1 \le d \le
\delta$, comme annoncé. Le lecteur attentif remarquera que ce résultat de nullité de
$\Delta_{1,d}(\uQ)$ dans la fourchette est encore valide
pour:
$$
\uQ = (a_1X_1 + a_3X_3,\ b_1X_1 + b_3X_3,\ P_3)   \qquad \deg(P_3) = e
$$
Le cas $d \ge \delta+1$ attendra que l'on établisse l'égalité $\Delta_{1,d}(\uQ) = \Res(\uQ)$.

\subsubsection*{La forme $\mu_d : \bigwedge\nolimits^{\chi_d}(\bfA[\uX]_d) \to \bfA$
associée à un système d'orientations de la décomposition de Macaulay (épisode II)}

On reprend le contexte précédant la
proposition~\ref{DeltakdDiviseMineurs}: un système monomial
d'orientations de la décomposition de Macaulay de $\rmK_{\sbullet,d}$,
etc. Rappelons que $\chi_d = \dim \Smac_{0,d}$.  La
proposition~\ref{CayleyDetVectoriel}, dans le cadre du complexe 
composante homogène de degré $d$ du complexe de Koszul de $\uP$,
fournit une forme $\mu_d : \BW^{\chi_d}(\bfA[\uX]_d) \to \bfA$, forme
déterminant de Cayley bien précise de $\rmK_{\sbullet,d}(\uP)$.  La
proposition~\ref{muDelta2Expression} \og expression d'un générateur de
$\CayleyVect(F_\sbullet)$ \fg{}, établie dans le cadre général d'un
complexe de Cayley $F_\sbullet$, permet
d'énoncer dans notre contexte-Koszul:

\begin {prop}
\label{mudDelta2dExpression}
Tout système monomial d'orientations de la décomposition de Macaulay de~$\rmK_{\sbullet,d}$
donne naissance à la forme $\mu_d$ définie à partir de $\Theta_{1,d} \in \BW^{\chi_d}(\bfA[\uX]_d)$
par:
$$
\mu_d = \Theta_{1,d}^\sharp \overset{\rm def.}{=}
\oriented{\Theta_{1,d} \,\wedge \,\sbullet\,}_{\bfe_{\Jex_{1,d}} \wedge \bfe_{\Smac_{0,d}}}
$$
Elle vérifie:
$$
\big(\BW^{s_d}(\Syl_d)(\bfe_{S_{1,d}})\big)^\sharp = \Delta_{2,d}\,\mu_d
\qquad \text{symboliquement} \qquad 
\mu_d = \dfrac{\big(\BW^{s_d}(\Syl_d)(\bfe_{S_{1,d}})\big)^\sharp}{\Delta_{2,d}}
$$
\end {prop}

Nous allons nous concentrer sur l'égalité à gauche ci-dessus, principalement en degré $d \ge \delta$.
Nous la ré-écrivons sous la forme:
$$
\Big[\BW^{s_d}(\Syl_d)(\bfe_{\Smac_{1,d}}) \wedge \sbullet \ \Big]_{\bfe_{\Jex_{1,d}}\wedge\, \bfe_{\Smac_{0,d}}}
= \Delta_{2,d} \ \mu_d(\sbullet)
\leqno (\clubsuit)
$$
Essayons de déjouer cette écriture un tantinet cryptique dans laquelle
$s_d = r_{1,d} = \dim \Jex_{1,d}$.  Les deux formes de part et d'autre
de l'égalité ne dépendent que de l'orientation~$\bfe_{\Smac_{0,d}}$: à
gauche, cela est dû au fait que $\bfe_{\Jex_{1,d}}$ et
$\bfe_{\Smac_{1,d}}$ se correspondent par $\BW^{s_d}(\varphi)$ et à
droite cela résulte du point ii) du lemme \ref{MultiplicateursOrientationsDecomposition}.
Il y a une certaine redondance dans ces deux justifications puisque le scalaire
$\Delta_{2,d}$ lui est indépendant de toute orientation.


\subsubsection*{L'égalité $(\clubsuit)$ en degré $d \geqslant \delta+1$}

En degré $d \geqslant \delta + 1$,  cette égalité prend le visage suivant 
$$
\det B_{1,d} \ = \ 
\Delta_{2,d} \times \Delta_{1,d}
$$
qui est, comme $B_{1,d} = W_{1,d}$, un avatar de l'égalité suivante (cf.~\ref{MacaulayParRecurrence}) :
$$
\det W_{1,d} \ = \ 
\det W_{2,d} \times \calR_d
$$
Pour comprendre l'analogie de ces deux formules, une première chose à
noter est l'égalité $\Delta_{1,d} = \calR_d$, c'est ce que nous allons
montrer ci-dessous.  La preuve (indépendante de la
section~\ref{sousSectionMacaulayRecurrence}) n'a pas lieu par
récurrence sur le nombre de variables, et utilise exclusivement le
degré $d$, contrairement à la formule~\ref{MacaulayParRecurrence}
sortie, comme par miracle, on ne sait d'où. Cette nouvelle formule
constitue en quelque sorte la voie normale dans l'obtention d'un
déterminant de Cayley \emph {normalisé} du complexe
$\rmK_{\sbullet,d}$.

\begin {prop}
\label{mu_d=R_d}
Soit $d \geqslant \delta+1$.

Pour tout système $\uP$:
$$
\Delta_{1,d}(\uP) = \calR_d(\uP)
$$
Autrement dit, avec les précautions d'usage décrites auparavant concernant les quotients, on a, 
$$
\calR_d(\uP) \ = \ 
\dfrac{\det B_{1,d} \,\det B_{3, d} \, \cdots} {\det B_{2, d}\,\det B_{4, d}\,\cdots}
= \dfrac{\det B_{1,d}}{\Delta_{2,d}}
\qquad \text{où} \qquad
\Delta_{2,d}  \ \overset{\rm def}{=} \ 
\dfrac{\det B_{2,d}\,\det B_{4, d}\,\cdots} {\det B_{3, d}\,\det B_{5, d}\,\cdots}
$$

\end {prop}

\begin {proof} On peut supposer $\uP$ générique au dessus de $\bbZ$.

\smallskip
Puisque $\calR_d(\uP)$ et $\Delta_{1,d}(\uP)$ sont des invariants de MacRae de $\bfB_d$,
il y a un inversible $\varepsilon$ de l'anneau des coefficients de $\uP$ tel que
$\Delta_{1,d}(\uP) = \varepsilon\,\calR_d(\uP)$. Cet anneau 
étant un anneau de polynômes sur~$\bbZ$, on a $\varepsilon = \pm 1$.
On termine en spécialisant $\uP$ en le jeu étalon $\uX^D$ pour lequel
$\calR_d(\uX^D) = \Delta_{1,d}(\uX^D) = 1$.

\end {proof}  

\noindent
\textbf{Commentaire.}

Ce résultat et le commentaire qui précède l'énoncé permettent d'en
déduire l'égalité $\Delta_{2,d} = \det W_{2,d}$ pour $d \geqslant
\delta+1$.  Pour l'instant, la preuve de cette égalité utilise le
résultat~\ref{MacaulayParRecurrence} dont nous voulons nous passer!
Une preuve indépendante et plus générale (pour tout
$d$) est contenue dans le théorème~\ref{DeltakdProdBinomialDetWhd}.

\subsubsection*{L'égalité $(\clubsuit)$ en degré $d = \delta$}

En degré $\delta$, la caractéristique d'Euler-Poincaré $\chi_\delta$ vaut $1$, 
et on a $\Smac_{0,\delta} = \bfA X^\emouton$.
On considère un système quelconque d'orientations de la décomposition 
$\rmK_{\sbullet,\delta} = \Mmac_{\sbullet,\delta} \oplus \Smac_{\sbullet,\delta}$, 
en imposant cependant d'orienter $\Smac_{0,\delta} = \bfA X^\emouton$
par~$X^\emouton$. 
La formule trèfle s'écrit alors :
$$
\omega = \Delta_{2,\delta} \times \mu_\delta
$$
qui est un avatar de l'égalité (cf.~\ref{MacaulayParRecurrenceDegreeDelta}) :
$$
\omega = \det W_{2,\delta} \times \omega_\res
$$

\begin {prop} [En degré $\delta$]
\label{mu_delta=omegares}  
\leavevmode

Pour tout système $\uP$, on a l'égalité $\mu_\delta = \omegares$. 
Autrement dit :
$$
\omegares \ = \ 
\dfrac{\omega}{\Delta_{2,\delta}}
\qquad \text{où} \qquad
\Delta_{2,\delta} \ \overset{\rm def}{=} \
\dfrac{\det B_{2,\delta}\,\det B_{4, \delta}\,\cdots} {\det B_{3, \delta}\,\det B_{5, \delta}\,\cdots}
$$
A fortiori $\Delta_{1,\delta} = \omegares(X^\emouton)$.  

\end {prop}

\begin {proof} \leavevmode

On peut supposer le système $\uP$ générique au dessus de $\bbZ$. 
Notons $\mu_{\delta,\uP} : \bfA[\uX]_\delta \to \bfA$ au lieu de $\mu_\delta$.

\medskip
Il y a un inversible $\varepsilon$ de l'anneau des coefficients de $\uP$ tel que
$\mu_{\delta,\uP} = \varepsilon\,\omegaRes{\uP}$. Cet anneau des coefficients étant un anneau
de polynômes sur $\bbZ$, on a $\varepsilon = \pm 1$. Pour montrer que $\varepsilon = 1$,
il suffit de montrer que les évaluations des deux formes linéaires en $X^\emouton$ sont
les mêmes.  L'évaluation en $X^\emouton$ conduit à:
$$
\Delta_{1,\delta}(\uP) = \varepsilon\, \omegaRes{\uP}(X^\emouton)
$$
Dans cette égalité, on peut spécialiser $\uP$ en le jeu étalon $\uX^D$ sans toucher à $\varepsilon = \pm 1$.
A gauche, cette spécialisation donne $\Delta_{1,\delta}(\uX^D) = 1$. Idem à droite du côté de $\omegares$
puisque $\omegaRes{\uX^D} = (X^\emouton)^\star$. D'où~\mbox{$\varepsilon = 1$}.
\end {proof}  

\noindent

De nouveau, nous signalons que la preuve ci-dessus n'utilise \emph
{pas} l'égalité $\omega = \det W_{2,\delta}\,\omega_\res$ rappelée
avant la proposition.  Cette égalité montre que l'on a
$\Delta_{2,\delta} = \det W_{2,\delta}$ et nous conduit, quitte à
paraître un tantinet répétitifs, à apporter la précision suivante,
analogue à celle du cas $d \ge \delta+1$: le résultat général
$\Delta_{2,d} = \det W_{2,d}$ pour tout $d$ est conséquence du
théorème ultérieur~\ref{DeltakdProdBinomialDetWhd}, à l'aide d'une
preuve complètement indépendante de la
section~\ref{sousSectionMacaulayRecurrence}.

\subsubsection*{Un aperçu des relations binomiales}

En principe, à cet endroit, le lecteur doit avoir compris que dans le
chapitre~\ref{ChapBW} figure une preuve (de nature combinatoire) de
l'égalité générale $\Delta_{2,d} = \det W_{2,d}$. Cette égalité traduit
donc que $\Delta_{2,d}$, initialement \emph {quotient} alterné de
déterminants, est égal à un unique déterminant.

\smallskip

Dans le même ordre d'idées, pour $k \geqslant 3$, le scalaire
$\Delta_{k,d}$ s'exprime comme un \emph{produit} de déterminants
$(\det W_{h,d})^{f_{k,h}}$ pour $h \geqslant k$ où $f_{k,h}$ est un
certain coefficient binomial.  Nous complétons maintenant le
paragraphe intitulé ``Identification du cofacteur''
(section~\ref{CofacteursTordus},
page~\pageref{IdentificationCofacteurCommentaire}) dans lequel nous
avons commencé à mentionner l'histoire des relations binomiales.
Comme annoncé dans ce paragraphe, les relations binomiales s'observent
d'abord au niveau des endomorphismes $B_{k,d}$ et $W_{h,d}$ sous la
forme précise suivante (cf le théorème
\ref{DetBkdProdBinomialDetWhd}):
$$
\forall\, k \geqslant 1, \quad 
\det B_{k,d} \ = \ 
\prod_{h = k}^n \ \big(\det W_{h,d}\big)^{e_{k,h}} 
\qquad 
\text{où \  $e_{k,h} = \binom{h-2}{h-k}$} 
$$
Il est important de rappeler la convention binomiale que nous avons
adoptée (cf. le chapitre~\ref{ObjetsSuiteP}, avant la
proposition~\ref{dminKk}).  Ainsi pour $k=1$, on a $e_{1,h} = 0$ pour
tout $h\geqslant 2$ et $e_{1,h} = 1$ pour $h=1$, si bien que
l'égalité ci-dessus n'est autre que $\det B_{1,d} = \det W_{1,d}$.
Ce n'est pas une information très intéressante puisque nous savons
que $B_1 = W_1$ mais nous constatons que la formule exprimant $\det B_{k,d}$
a l'avantage d'être uniforme en~$k$.

\smallskip

Et ces dernières relations entraînent les relations binomiales dont il
était question précédemment entre les $\Delta_{k,d}$ et les $\det
W_{h,d}$ (cf. le théorème \ref{DeltakdProdBinomialDetWhd}):
$$
\forall\, k \geqslant 1, \quad 
\Delta_{k,d} \ = \ 
\prod_{h = k}^n \ \big(\det W_{h,d}\big)^{f_{k,h}}
\qquad 
\text{où \  $f_{k,h} = \binom{h-3}{h-k}$} 
$$
Encore une fois, il faut tenir compte de la convention binomiale
adoptée.  Pour $k = 2$, on a $f_{2,h} = 0$ pour tout $h \geqslant 3$
et $f_{2,h} = 1$ pour $h = 2$; pour $k = 3$, on a $f_{3,h} = 1$ pour
tout $h \geqslant 3$. Si bien que pour $k=2,3$, l'expression de
$\Delta_{k,d}$ est la suivante:
$$
\Delta_{2,d} = \det W_{2,d}, \qquad\qquad
\Delta_{3,d} = \det W_{3,d} \det W_{4,d} \cdots \det W_{n,d}
$$
Le cas $k=1$ est un peu spécial, cf l'énoncé du théorème
\ref{DeltakdProdBinomialDetWhd} et le commentaire qui vient après: les
exposants non nuls $(f_{1,h})_{h \ge 1}$ sont $f_{1,1}=1$, $f_{1,2}
= -1$ et l'expression obtenue pour $\Delta_{1,d}$
est le quotient (exact) $\det W_{1,d} / \det W_{2,d}$.  C'est bien sûr
un polynôme en les coefficients des $P_i$ mais visiblement \emph{cette}
expression de $\Delta_{1,d}$ n'est pas polynomiale en les
$\det W_{h,d}$.

En pensant au fait que $\Delta_{1,d}$ n'est autre que le résultant
pour $d \geqslant \delta+1$, nous nous sommes toujours dits que
l'absence pour $\Delta_{1,d}$ d'une formule déterminantale n'avait
rien d'étonnant.  Rien de formel dans cette dernière phrase qui obéit
plutôt au principe: s'il existait, pour le résultant, une formule
magique de type $\Res(\uP) = \det W_\calM(\uP)$, \og cela se saurait\fg.

\medskip

Présentons ces deux jeux de relations binomiales pour $n=6$.
Supprimons l'indice $d$ sur les objets (ici~$d$ est un entier quelconque)
et posons $b_k = \det B_{k,d}$ et $w_h = \det W_{h,d}$.
On peut résumer les relations binomiales donnant $b_k$ en fonction des $(w_h)_{h\geqslant k}$
de manière schématique par les tableaux suivants qui se lisent en ligne.
On commence par construire le tableau de gauche donnant les \emph{exposants}~$e_{k,h}$.
On en déduit les exposants~$f_{k,h}$ du tableau de droite 
grâce à la définition de~$\Delta_k$ comme quotient alterné des~$b_k$.
Ainsi la ligne de~$\Delta_k$ s'obtient 
en faisant (exposants obligent) la somme alternée des $k$ dernières lignes du tableau de gauche.
On a par exemple $b_4 = w_4 w_5^3 w_6^6$ et $\Delta_4 = w_4 w_5^2 w_6^3$.
$$
\renewcommand{\arraystretch}{1.5}
\begin{tabular}[t]{c|*{6}c}
& $w_1$ & $w_2$ & $w_3$ & $w_4$ & $w_5$ & $w_6$ \\
\hline 
$b_1$ & 1 &  & & & & \\
$b_2$ & & 1 & 1 & 1 & 1 & 1 \\
$b_3$ & &   & 1 & 2 & 3 & 4 \\	 
$b_4$ & &   &   & 1 & 3 & 6\\ 
$b_5$ & &   &   &  & 1 & 4\\ 
$b_6$ & & & & & & 1\\ 
\end{tabular}
\qquad \qquad \qquad
\begin{tabular}[t]{c|*{5}c}
& $w_2$ & $w_3$ & $w_4$ & $w_5$ & $w_6$ \\
\hline 
$\Delta_2$ & 1 & & & &  \\
$\Delta_3$ & & 1 & 1 & 1 & 1 \\
$\Delta_4$ & &   & 1 & 2 & 3 \\	 
$\Delta_5$ & &   &   & 1 & 3 \\ 
$\Delta_6$ & &   &   &  & 1  \\ 
\end{tabular}
$$
Avec ces quelques lignes, nous espérons avoir convaincu le lecteur du 
rôle fondamental joué par les déterminants des endomorphismes $B_{k,d}$, $W_{h,d}$
et par la famille de scalaires $\Delta_{k,d}$.

\subsubsection*{Les quotients alternés chez Demazure,
                \textit{Une définition constructive du résultant} \cite{Demazure1}}

Comme nos notations diffèrent de celles de Demazure, reprécisons certains choix.
Celui de l'indice $k$ pour l'endomorphisme $B_k$ est dû au fait
qu'il est issu de la différentielle $\partial_k(\uP)$ du complexe de
Koszul de~$\uP$.  Notre choix de la lettre~$B$ n'est pas profond et ne
correspond à aucun acronyme: c'est simplement une allusion à la
position de cette lettre dans une matrice $\left[\begin {smallmatrix}
    A & B\cr C & D\cr\end {smallmatrix}\right]$!
Le point fondamental est que $B_k$ est un \emph {endomorphisme} obtenu
en \emph {normalisant} l'induit-projeté $\beta_k$ de la différentielle $\partial_k(\uP)$,
qui se trouve en \og position $B$\fg{} dans la matrice $\left[\begin
    {smallmatrix} \alpha_k & \beta_k\cr \gamma_k & \delta_k\cr\end {smallmatrix}\right]$,
cf la définition générale \ref{ContexteDetCayley} et celle plus particulière~\ref{EndoBk}.
Cette dernière matrice est relative à la somme directe $\rmK_\sbullet = \Mmac_\sbullet \oplus
\Smac_\sbullet$ et l'expression ``endomorphisme pour $B_k$'' est à comprendre comme $\bfA$-endomorphisme
$\bbN^n$-gradué du sous-module $\Mmac_{k-1}$ (quitte à nous répéter, attention
au décalage d'une unité $k \leftrightarrow k-1$).

\medskip

Dans \cite{Demazure1}, nos $\det
B_{k,d}(\uP)$ sont notés~$D_k(\uP,d)$.  Voici ce que dit Demazure :
\og{} On obtient le résultant comme produit alterné de
déterminants~$D_{k,d}$, $k = 1, \dots, n$ \fg{}.

Concernant les quotients alternés, l'indexation de Demazure 
est régie par l'égalité 
$$
D_{k,d} \ =\  \Delta_k^d \,\Delta_{k-1}^d
\qquad\qquad
\hbox {\footnotesize (Demazure, Une définition constructive du résultant)}
$$
Il y a donc un décalage d'une unité par rapport à notre choix :
$$
\det B_{k,d} \ =\  \Delta_{k,d}\, \Delta_{k+1,d} 
$$
Ainsi, son $\Delta_0^d$ est notre $\Delta_{1,d}$ et son $\Delta_1^d$ est notre $\Delta_{2,d}$.  

\`A la fin de son introduction, Demazure mentionne (en 1984)
\og La définition des $D_1$ et~$\Delta_0$ et le calcul des déterminants $D_1$ en termes des
$\Delta_0$ se trouvent déjà, le 8 mai 1902, dans l'article de Macaulay
\textit{Some formulae in elimination}, article profond et méconnu \fg{}.
On ne voit pas, dans l'article de Demazure, d'égalité qui correspondrait à 
$$
\Delta_{2,d} \ =\ \det W_{2,d}
\qquad\qquad
\hbox {\footnotesize (on pourrait s'attendre à la trouver en degré $d \geqslant \delta+1$)}.
$$
que nous démontrerons en~\ref{DeltakdProdBinomialDetWhd} pour n'importe quel $d$.

\subsection{La décomposition de Macaulay tordue par une permutation $\sigma\in\fS_n$}
\label{sigmaDecompositionMacaulay}

On veut maintenant définir une $\sigma$-décomposition de Macaulay 
qui généralise la décomposition $\bfA[\uX] = \Jex_1 \oplus 
\bigoplus_{\alpha \preccurlyeq D-\mathds 1} \bfA X^\alpha$ 
à l'algèbre extérieure $\bigwedge\big(\bfA[\uX]^n \big)$
et qui fasse en sorte que la première fenêtre d'observation s'identifie
à l'endomorphisme $W_1^{\sigma}$ défini en~\ref{SigmaVersion}.
Pour réaliser la première condition (à savoir la décomposition en somme directe), on prend $\Mmac_0^\sigma := \Jex_1$. 
Par conséquent, en degré homologique $0$, rien n'est modifié par rapport 
à la définition~\ref{DefDecompositionMacaulay} : on a 
$\Mmac_0^\sigma = \Mmac_0$ et $\Smac_0^\sigma = \Smac_0$.

\smallskip

Concernant la deuxième condition (obtention d'une fenêtre
d'observation s'identifiant à $W_1^\sigma$), la définition de
$W_1^\sigma$ nous invite à définir le module $\Smac^\sigma_1$ comme
étant le $\bfA$-module libre de base les $(X^\alpha/X_i^{d_i})\,e_i$
où $X^\alpha \in \Jex_1$ et $i = \sminDiv(X^\alpha)$.  En utilisant
l'ordre $<_\sigma$ défini en section \ref{SigmaVersion} et la
convention $\sminDiv(X^\gamma) = n+1$ si $\DivSeq(X^\gamma) =
\emptyset$, ces monômes extérieurs s'écrivent $X^\gamma e_i$ avec
$\sminDiv(X^\gamma) \geqslant_{\sigma} i$.  En considérant les monômes
extérieurs de $\rmK_1$ non dans~$\Smac_1^\sigma$, on obtient le module
$\Mmac_1^\sigma$ : c'est le $\bfA$-module de base les
$X^\beta e_j$ avec $\sminDiv(X^\beta) <_\sigma j$ (ce qui force
$X^\beta \in \Jex_1$), ou encore le $\bfA[\uX]$-module engendré par
les $X_i^{d_i} e_j$ avec $i <_\sigma j$.

\medskip

Il est naturel de généraliser cela en degrés homologiques supérieurs de la façon suivante :
le module~$\Mmac_k^\sigma$ est alors le $\bfA$-module de base les $X^\beta e_J$ 
tels que $\sminDiv(X^\beta) <_\sigma \smin J$, 
ou encore le $\bfA[\uX]$-module engendré par 
les $X_i^{d_i} e_J$ avec $i <_\sigma \smin J$.
On retrouve donc une définition analogue à celle de~\ref{DefDecompositionMacaulay}
en changeant l'ordre $<$ sur $\{1..n\}$ en~${<_\sigma}$.

\begin{defn}\label{DefMksigmaSksigma} \leavevmode

On désigne par $\Mmac_k^\sigma$ le sous-$\bfA$-module de $\rmK_k$
de base les~$X^\beta e_J$ avec $X^\beta \in \langle X_i^{d_i} \mid i
<_\sigma J \rangle$ et $\#J = k$.  
C'est un sous-$\bfA[\uX]$-module de type fini
et l'appartenance $X^\beta e_J \in \Mmac_k^\sigma$ 
peut-être réalisée par $\sminDiv(X^\beta) <_{\sigma} J$ 
(nécessairement ${X^\beta \in \Jex_1})$.

\smallskip
On note $\Smac_k^\sigma$ le supplémentaire monomial de $\Mmac_k^\sigma$ dans $\rmK_k$.
Il admet pour $\bfA$-base les monômes extérieurs non dans $\Mmac_k^\sigma$ 
de sorte que $\rmK_k = \Mmac_k^\sigma \oplus \Smac_k^\sigma$.

\smallskip
Pour $k=0$, on a 
$\bfA[\uX] =\Mmac_0^\sigma \oplus \Smac_0^\sigma$
avec $\Mmac_0^\sigma = \Mmac_0 = \langle X_1^{d_1}, \dots, X_n^{d_n}\rangle$ 
et $\Smac_0^\sigma = \Smac_0 = \bigoplus_{\alpha \preccurlyeq D-\mathds 1} \bfA X^\alpha$.
Pour ${k=n}$, on a $\Mmac^\sigma_n = 0$.
\end{defn}

\label {NOTA14-MmacSmacsigma}%
%
%

\begin{rmq}
Décrivons le module $\Mmac_k^\rho$ où $\rho : i \mapsto n-i+1$
est l'involution renversante de $\{1..n\}$.
Il est clair que $\rhomin(E) = \max(E)$ et que l'ordre $<_\rho$ est exactement l'ordre $>$ usuel.
On a alors l'équivalence
$$
X^\beta e_J \in \Mmac_k^\rho\ \Leftrightarrow \ \maxDiv(X^\beta) > \max J
$$
\end{rmq}

\begin{defn}
L'application $\bfA$-linéaire 
$\beta_k^\sigma(\uP) : \Smac_k^\sigma \rightarrow \Mmac_{k-1}^\sigma$ 
est l'induit-projeté de $\partial_k(\uP)$, 
\idest{} $\beta_k^\sigma(\uP) = \pi_{\Mmac_{k-1}^\sigma} \circ \partial_k(\uP) \circ \iota_{\Smac_k^\sigma}$.
\end{defn}

Les définitions et propriétés sont pratiquement inchangées par rapport
à~\ref{DecompositionMacaulay}.  Il faut juste faire attention au fait
que $\smin I$ n'est pas inférieur (au sens classique) aux éléments
de $I$. Il s'introduit donc des signes dans les différentielles.  On
peut prendre en charge ces signes en utilisant, pour $i \in I$, le produit
intérieur droit $e_I \intd e_i$ égal à $(-1)^{\varepsilon_i(I)} e_{I
  \setminus i}$ où $\varepsilon_i(I)$ est le nombre d'éléments de $I$
strictement inférieurs à~$i$ (pour l'ordre habituel $<$ sur les
entiers). Ceci a pour effet que l'isomorphisme $\varphi^\sigma$
de la proposition suivante n'est pas nécessairement monomial
comme dans le cas où $\sigma = \id$: l'image d'un monôme extérieur
par $\varphi^\sigma$ est un monôme extérieur \emph {au signe près}.

\begin{prop}
Pour le jeu étalon $\uX^D$, l'application $\beta_k^\sigma(\uX^D) : \Smac_k^\sigma \rightarrow \Mmac_{k-1}^\sigma$ réalise
$X^\alpha e_I  \, \mapsto\,  X^\alpha X_i^{d_i} e_I \intd e_i$ 
où $i = \smin I$. C'est un isomorphisme de $\bfA$-modules dont l'inverse $\varphi^\sigma$ est donné par 
$$
\varphi^\sigma : 
\begin{array}[t]{rcl}
\Mmac_{k-1}^\sigma & \longrightarrow & \Smac_k^\sigma \\ [0.3em]
X^\beta e_J & \longmapsto & 
\dfrac{X^\beta}{X_i^{d_i}} e_i \wedge e_J \quad 
\text{où $i = \sminDiv(X^\beta)$}
\end{array}
$$
\end{prop}

\begin{proof}
On laisse le soin au lecteur de vérifier la formule donnée pour $\beta_k^\sigma(\uX^D)$
en s'inspirant de la preuve de la proposition~\ref{varphiIso}.
Pour montrer que les applications $\beta^\sigma(\uX^D)$ et $\varphi^\sigma$ sont inverses l'une de l'autre, 
il suffit d'utiliser les deux formules 
$$
e_i \wedge (e_I \intd e_i) = e_I\ \text{pour $i \in I$ }
\qquad \text{et } \qquad
(e_i \wedge e_J) \intd e_i = e_J \ \text{pour $i \notin J$}
$$
ainsi que les deux implications :
$$
\sminDiv(X^\alpha) \geqslant_\sigma i \ \Rightarrow \ 
\sminDiv\big(X^\alpha X_i^{d_i} \big) = i
\qquad \qquad \text{et} \qquad \qquad 
i <_\sigma J \ \Rightarrow \ 
\smin (i \vee J) = i
$$
\end{proof}

\label {NOTA14-varphisigma}%

\medskip

Les résultats relatifs à la décomposition de Macaulay ``ordinaire'', \idest{}
celle pour $\sigma = \id$, s'adaptent aisément à la
décomposition de Macaulay tordue par n'importe quel $\sigma \in \fS_n$.
Dans les sections ultérieures, nous en fournirons quelques uns sans preuve
comme par exemple celui qui suit.

\begin{defn}\label{DefBksigma}
Le $\bfA$-endomorphisme $B_k^\sigma(\uP)$ de $\Mmac^\sigma_{k-1}$
est défini par $B_k^\sigma(\uP) = \beta_k^\sigma(\uP) \circ \varphi^\sigma$.
Pour $X^\beta e_J \in \Mmac^\sigma_{k-1}$ avec $i = \sminDiv(X^\beta)$, on a
$$
\big(B_k^\sigma(\uP)\big) (X^\beta e_J)
\ = \ 
\pi_{\Mmac^\sigma_{k-1}} 
\Bigg( 
\dfrac{X^\beta}{X_i^{d_i}}\  P_i \ e_J 
\  + \ 
\sum_{j \in J} 
\dfrac{X^\beta}{X_i^{d_i}} \ 
P_j \ (e_i \wedge e_J) \intd e_j
\Bigg)
$$
En particulier, cet endomorphisme est normalisé au sens où pour le jeu étalon, 
on a $B_k^\sigma(\uX^D) = \Id$.
De plus, l'endomorphisme $B_1^\sigma$ de $\Mmac_0$ 
coïncide avec l'endomorphisme $W_1^\sigma$ de $\Jex_1 = \Mmac_0$.
\end{defn}

\index{B@les endomorphismes!2@$B^\sigma_k(\uP)$ de $\Mmac^\sigma_{k-1}$ où $\sigma\in\fS_n$}%
\index{B@les endomorphismes!2@$B^\sigma_{k,d}(\uP)$ de $\Mmac^\sigma_{k-1,d}$}%

\subsubsection{Les scalaires $\Delta^\sigma_{k,d}(\protect\uP)$ pour $\sigma\in\fS_n$}

A partir de la $\sigma$-décomposition de Macaulay
$(\Mmac^\sigma_{\sbullet,d}, \Smac^\sigma_{\sbullet,d})$ de
$\rmK_{\sbullet,d}(\uX^D)$ que nous venons de mettre en place,
ces scalaires sont définis de manière analogue aux $\Delta_{k,d}(\uP)$.
Le résultat suivant est le pendant de la proposition~\ref{DeltakdDiviseMineurs}
et, en ce qui concerne sa justification, le lecteur pourra consulter
le texte qui précède la proposition référencée.

\begin{prop} \label{DeltakdTorduDiviseMineurs}
Le scalaire $\Delta_{k+1,d}^\sigma$ divise tous les mineurs d'ordre $r_{k,d}$ de $\partial_{k,d}$
indexés par les colonnes de $\Smac_{k,d}^\sigma$.
En particulier, $\Delta_{2,d}^\sigma$ divise tous les mineurs d'ordre $s_d$ de 
$\Syl_d$ indexés par les colonnes de $\Smac_{1,d}^\sigma$.
\end{prop}

Voici un résultat important spécifique à $\Delta^\sigma_{1,d}$ qui n'a
pas d'équivalent pour les $\Delta^\sigma_{k,d}$ lorsque $k \ge 2$.

\begin {lem}
\label{EgaliteDelta1dDelta1dsigma}
Pour $\sigma \in \fS_n$ et $d \in \bbN$, on a $\Delta^\sigma_{1,d}(\uP) = \Delta_{1,d}(\uP)$.
\end {lem}

\begin {proof}
On dispose de deux décompositions du complexe $\rmK_{\sbullet,d}(\uP)$: celle de Macaulay et celle
de Macaulay tordue par $\sigma$. Elles coïncident en degré homologique 0 puisqu'égales à
$\rmK_{0,d} = \Jex_{1,d} \oplus \Smac_{0,d}$, ce qui permet d'appliquer le lemme de
rigidité~\ref{RigidityLemma}.
\end {proof}

\subsubsection*{L'exemple $D=(4,4,4)$, $d=8$ permettant d'illustrer les $\sigma$-décompositions de Macaulay}

Nous avons choisi cet exemple qui nous semble pertinent pour plusieurs raisons. D'une part,
$d=8$ est juste en dessous de $\delta=9$ et d'autre part, on a $d = d_i + d_j$ de sorte que
$(e_{12}, e_{13}, e_{23})$ est une base~$\rmK_{2,d}$. Ce dernier module est donc de dimension 3.
Voici d'autres dimensions utiles:
$$
\dim \bfA[\uX]_d = \binom{n+d-1}{d} = \binom{10}{8} = \binom{10}{2} = 45
$$
La matrice de Sylvester $\Syl_d(\uP)$ est carrée puisque:
$$
\binom{n+d-d_i-1}{d-d_i} = \binom{3+4-1}{4} = \binom{6}{4} = 15,
\qquad
\dim \rmK_{1,d} = \sum_{i=1}^n \binom{n+d-d_i-1}{d-d_i} = 3 \times 15 = 45
$$
Comme $d=\delta-1$, la base monomiale du supplémentaire monomial $\Smac_{0,d}$
de $\Jex_{1,d}$ est constituée des $X^\emouton/X_i$ pour $i=1,2,3$.

\smallskip

Dans le schéma, chaque nombre indique la dimension de l'espace en dessous ou en dessus:
$$
\def \egal{\rotatebox{90}{=}}
\xymatrix @H=0pt @R=2pt @C=4pc {
3                     &45                                  &45  \\  
\rmK_{2,d}\ar[r]^-{\partial_{2,d}(\uP)}  &\rmK_{1,d}\ar[r]^-{\partial_{1,d}(\uP)} &\rmK_{0,d} = \bfA[\uX]_d  \\
\egal                 &\egal                                &\egal \\
\Mmac_{2,d}=0  &\Mmac_{1,d}\ar[ddl]_{\textstyle\varphi}^{\simeq} &\Mmac_{0,d}=\Jex_{1,d}\ar[ddl] _{\textstyle\varphi}^{\simeq} \\
\oplus                &\oplus                               &\oplus \\
\Smac_{2,d}            &\Smac_{1,d}                          &\Smac_{0,d} \\
(e_{12},e_{13},e_{23})  &42                                   &(X_1^3X_2^3X_3^2,\,X_1^3X_2^2X_3^3,\,X_1^2X_2^3X_3^3) \\
}
$$
La base monomiale de $\Mmac_{1,d}$ est constituée des 3 monômes $X_i^4 e_j$ avec $i < j$
et $\varphi(X_i^4 e_j) = e_{ij}$.

\medskip

Examinons les $\sigma$-décompositions tordues par $\sigma \in \fS_3$. Tout d'abord,
$\Smac^\sigma_{0,d} = \Smac_{0,d}$ pour tout $\sigma$. Quant à $\Mmac^\sigma_{1,d}$, sa base
monomiale est formée des $X_i^4e_j$ avec $i <_\sigma j$ où l'on  rappelle que
$\sigma(1) <_\sigma \sigma(2) <_\sigma \sigma(3)$. Prenons par exemple $\sigma = (1,2,3)$:
$$
2 <_\sigma 3 <_\sigma 1, \qquad
\xymatrix {
(X_2^4e_1, X_3^4e_1, X_2^4e_3) \ar[r]^-{\textstyle\varphi^\sigma} & (-e_{12},\  -e_{13}, e_{23})
}
$$
Un autre exemple pour $\sigma = (1,3,2)$:
$$
3 <_\sigma 1 <_\sigma 2, \qquad
\xymatrix {
(X_3^4e_1, X_1^4e_2, X_3^4e_2) \ar[r]^-{\textstyle\varphi^\sigma} & (-e_{13},\  e_{12}, -e_{23})
}
$$
Faisons maintenant débarquer un système $\uP$ de format $D = (4,4,4)$ dont nous
numérotons les coefficients sur le modèle
$$
P_1 = a_1X_1^4 + a_2X_2^4 + a_3X_3^4 + \cdots
$$
Idem pour $P_2$ avec $b_\sbullet$ et $P_3$ avec $c_\sbullet$.

L'endomorphisme $B_{3,d}$ opère sur $\Mmac_{2,d}=0$, donc est de déterminant 1. Idem pour
la version tordue par $\sigma \in \fS_3$, donc $\Delta^\sigma_{3,d} = \det B^\sigma_{3,d} = 1$.
D'où:
$$
\det B^\sigma_{2,d} = \Delta^\sigma_{2,d} \Delta^\sigma_{3,d} = \Delta^\sigma_{2,d}
$$
Les $\Delta^\sigma_{2,d}$ se calculent donc comme le déterminant de
l'endomorphisme $B^\sigma_{2,d}$ de dimension 3.  Notons $\partial_2$
la différentielle du complexe de Koszul de $\uP$ définie par
$\partial_2(e_{ij}) = P_ie_j-P_je_i$.  Pour déterminer $B_{2,d}$ sur
les $X_i^4e_j$ de $\Mmac_{1,d}$, on calcule d'abord $\varphi(X_i^4e_j)
= e_{ij}$ puis
$$
\pi_{\Mmac_{1,d}} (P_ie_j-P_je_i)
$$
Par exemple:
$$
B_{2,d}(X_1^4e_2) = \pi_{\Mmac_{1,d}} (P_1e_2-P_2e_1) = a_1X_1^4 e_2, \quad
B_{2,d}(X_1^4e_3) = \pi_{\Mmac_{1,d}} (P_1e_3-P_3e_1) = (a_1X_1^4 + a_2X_2^4) e_3
$$
Idem pour les $\sigma$-versions. On voit ainsi que $B^\sigma_{2,d}$ ne dépend que des
coefficients en les $X_j^4$ des $P_i$.

Ainsi pour:
$$
\sigma = \Id_3,\qquad\qquad
\sigma = (1,2,3),\qquad\qquad
\sigma = (1,3,2)
$$
on trouve dans l'ordre pour les $B^\sigma_{2,d}$:
$$
\EastBordermatrix{
a_{1} & . & -c_{1} & \Heti{X_{1}^{4}\,e_{2}} \\ 
. & a_{1} & b_{1} & \Heti{X_{1}^{4}\,e_{3}} \\ 
. & a_{2} & b_{2} & \Heti{X_{2}^{4}\,e_{3}} \\ 
}
\qquad
\EastBordermatrix{
b_{2} & c_{2} & . & \Heti{X_{2}^{4}\,e_{1}} \\ 
b_{3} & c_{3} & . & \Heti{X_{3}^{4}\,e_{1}} \\ 
. & -a_{2} & b_{2} & \Heti{X_{2}^{4}\,e_{3}} \\ 
}
\qquad
\EastBordermatrix{
c_{3} & -b_{3} & . & \Heti{X_{3}^{4}\,e_{1}} \\ 
. & a_{1} & c_{1} & \Heti{X_{1}^{4}\,e_{2}} \\ 
. & a_{3} & c_{3} & \Heti{X_{3}^{4}\,e_{2}} \\ 
}
$$
Voici une \emph {recette} pour calculer $\Delta^\sigma_{2,d}$. On définit
$$
U = (u_{ij}) = \begin {bmatrix} a_1 & b_1 &c_1 \\ a_2 & b_2 &c_2 \\ a_3 & b_3 &c_3 \\ \end {bmatrix}
\qquad
\delta_{i,j} = u_{ii} \begin {vmatrix} u_{ii} & u_{ij} \\ u_{ji} & u_{jj} \\ \end {vmatrix}
\qquad
D_\sigma = \delta_{\sigma(1),\sigma(2)}
$$
Par exemple:
$$
D_{\Id_3} =
\delta_{1,2} = a_1 \begin {vmatrix} a_1 & b_1 \\ a_2 & b_2 \\ \end {vmatrix}
\qquad
D_{(1,2,3)} =
\delta_{2,3} = b_2 \begin {vmatrix} b_2 & c_2 \\ b_3 & c_3 \\ \end {vmatrix}
\qquad
D_{(1,3,2)} =
\delta_{3,1} = c_3 \begin {vmatrix} c_3 & a_3 \\ c_1 & a_1 \\ \end {vmatrix}
$$
On vérifie sur les 3 exemples que $\Delta^\sigma_{2,d} = D_\sigma$ et c'est également vrai pour les 3 transpositions
de $\fS_3$.

\bigskip

Quid de $\Delta_{1,d}$? On verra en \ref{poidsDelta1} que, de manière générale,  il est homogène en $P_i$ de poids
$\dim \Jex_{1\setminus2,d}^{(i)}$. Ici, dans le cadre du format $D = (4,4,4)$, ces dimensions sont indépendantes de $i$.
La base monomiale de $\Jex_{1\setminus2,d}^{(3)}$ est constituée des
$$
X_1^{i_1} X_2^{i_2} X_3^{d-(i_1+i_2)}  \qquad   0 \le i_1,i_2 \le 4-1, \quad  i_1+i_2 \le d-4
$$
Pour $d=8$, parmi les $(i_1,i_2) \in \{0,\dots,4-1\}^2$, il faut
enlever les 3 couples $(2,3), (3,2), (3,3)$ tels que $i_1 + i_2 > d-4$. Ce qui donne
$\dim\Jex_{1\setminus2,d}^{(3)} = 4^2 - 3 = 13$. Bilan: $\Delta_{1,d}$ est homogène de poids 13 en chaque~$P_i$.

\medskip

En ce qui concerne le calcul, de manière générale, on a $\det
W^\sigma_{1,d} = \Delta_{1,d} \Delta^{\sigma}_{2,d}$ ce qui permet de
déterminer $\Delta_{1,d}$ via plusieurs expressions indexées par
$\sigma \in \fS_n$:
$$
\Delta_{1,d} = \dfrac{\det W^\sigma_{1,d}}{\Delta^\sigma_{2,d}}
$$
\emph{Ici}, en générique, les $3! = 6$ dénominateurs sont distincts. Que dire des numérateurs?
Pas question de montrer ni $W^\sigma_{1,d}$ de dimension 42,  ni son déterminant y compris quand
les $P_i$ sont réduits à leurs termes en $X_j^4$ et les $a_\sbullet, b_\sbullet, c_\sbullet$ 
génériques. Par contre, en spécialisant beaucoup de coefficients à 0, par exemple:
$$
\uP = [X_1^4,\, X_2^4,\, X_3^4]\,U  \qquad \text{où} \qquad
U = \begin {bmatrix} 0 & b_1 &c_1 \\ 0 & b_2 &0 \\ a_3 & b_3 &c_3 \\ \end {bmatrix}
$$
on obtient à l'aide d'un système de Calcul Formel les expressions suivantes pour $\Delta_{1,d}$:
$$
\dfrac{0}{D_{\Id_3}=0} =
\dfrac{-a_3^{13}b_2^{15}c_1^{13}c_3}{D_{(1,2,3)}= b_2^2c_3} =
\dfrac{-a_3^{14}b_2^{13}c_1^{14}c_3}{D_{(1,3,2)}= -a_3c_1c_3} =
\dfrac{0}{D_{(2,3)}=0} =
\dfrac{0}{D_{(1,2)}=0} =
\dfrac{-a_3^{14}b_2^{14}c_1^{13}c_3^2}{D_{(1,3)}= b_2c_3^2}
$$
Les formes non indéterminées fournissent $\Delta_{1,d} = (-a_3b_2c_1)^{13} = \det(U)^{13}$, conforme au poids.
Et c'est ce que l'auteur souhaitait obtenir: la nullité de certains dénominateurs (en particulier pour
$\sigma = \Id_3$) et le fait que les autres soient distincts.

\smallskip

Ce qui n'était absolument pas prévu c'est que ce petit exemple
conduise à faire apparaître, pour tout système $\uP$ de format $D =
(q,q,\dots,q)$ de la forme
$$
\uP = [X_1^q,\, X_2^q,\, \cdots,\, X_n^q]\, U \qquad U \in \bbM_n(\bfA)
$$
une formule close pour $\Delta_{1,d}(\uP)$:
$$
\Delta_{1,d}(\uP) = \det(U)^{e_{q,d}} \qquad
\text{où $e_{q,d}$ est la dimension commune des $\Jex_{1\setminus2,d}^{(i)}(D)$}
$$
Et cette dimension commune $e_{q,d}$ est le nombre de points entiers du polytope de dimension $n-1$
défini par les inégalités:
$$
0 \le i_1 \le q-1,\  0 \le i_2 \le q-1,\ \dots,\   0 \le i_{n-1} \le q-1, \qquad
i_1 + \cdots + i_{n-1} \le d-q
$$
Pour $d \ge \delta(D)+1$, ce polytope est le cube
$[0,\dots,q-1]^{n-1}$ de dimension $n-1$ donc $e_{q,d} = q^{n-1}$.
On peut également utiliser le fait que pour $d \ge \delta(D) + 1$,
on a $\dim\Jex_{1\setminus2,d}^{(i)}(D) = \widehat d_i \overset{\rm ici}{=}
q^{n-1}$.

\smallskip

De manière plus générale, en remplaçant $(X_1^q, \dots, X_n^q)$ par un
système $\uQ$ où les $Q_i$ sont de même degré~$q$, nous disposons
d'une règle de transformation de $\Delta_{1,d}$ qui généralise celle
du résultant de la proposition \ref{PtimesU}.
Nous l'énonçons sans preuve et nous laissons le soin au
lecteur d'y \emph{réfléchir} en commençant par $U$ diagonale.

\begin {prop}
Soit $\uQ$ un système de format $D = (q,\dots,q)$ et $U \in \bbM_n(\bfA)$.
Notons $\uQ\,U$ le système de même format défini par $[Q_1,\dots,Q_n]\,U$.
Alors pour tout $d$, en conservant la notation~$e_{q,d}$:
$$
\Delta_{1,d}(\uQ\, U) = \det(U)^{e_{q,d}}\ \Delta_{1,d}(\uQ)  
$$
\end {prop}

\bigskip

Revenons à l'exemple en anticipant sur le prochain chapitre.
Dans la base $(X_1^4X_2^4, X_1^4X_3^4, X_2^4X_3^4)$ de~$\Jex_{2,d}$,
voici les $W^\sigma_{2,d}$ pour les trois permutations $\sigma$ utilisées
précédemment:
$$
\left[
\begin{array}{*{3}{c}}
a_{1}& .& . \\ 
.& a_{1}& b_{1} \\ 
a_{3}& a_{2}& b_{2} \\ 
\end{array}
\right]
\qquad
\left[
\begin{array}{*{3}{c}}
b_{2}& c_{2}& . \\ 
b_{3}& c_{3}& b_{1} \\ 
.& .& b_{2} \\ 
\end{array}
\right]
\qquad
\left[
\begin{array}{*{3}{c}}
a_{1}& c_{2}& c_{1} \\ 
.& c_{3}& . \\ 
a_{3}& .& c_{3} \\ 
\end{array}
\right]
$$
En ce qui concerne $W^\sigma_{2,d}$ versus $B^\sigma_{2,d}$, il y a quelque chose à remarquer: quoi?

\subsection {Structure triangulaire de $B_{k,d}(\protect\uQ)$ sur le jeu circulaire $\protect\uQ$ lorsque $k\geqslant 2$}
\label{sousSectionBkdJeuCirculaire}

L'objectif est de démontrer, pour $k \ge 2$ et pour tout $d$, que $\det B_{k,d}(\uQ) = 1$ où
$\uQ$ est le jeu circulaire mis en place dans le chapitre \ref{ChapJeuCirculaire} et défini
par $Q_i = X_i^{d_i} - X_i^{d_i-1} X_{i+1}$. Ceci permettra d'en déduire que $\Delta_{k,d}(\uQ) = 1$
puis d'en tirer, pour le jeu générique $\uP$, le fait que $\Delta_{k,d}(\uP)$
est régulier sur $\bfA[\uX]/\uPsat$.

\medskip

Cela va être commode de généraliser légèrement la définition
$B_{k,d}(\uP)$ en faisant intervenir n'importe quel sous-module
monomial $\calM$ de $\Mmac_{k-1,d}$. Chaque terme
$\rmK_k(\uP)$ du complexe de Koszul de~$\uP$ est un $\bfA$-module
libre de base les monômes extérieurs $X^\alpha e_I$ avec $\#I = k$,
ce qui permet de définir le projecteur~$\pi_\calN$ de $\rmK_k(\uP)$ pour
n'importe quel sous-module monomial $\calN$ de $\rmK_k(\uP)$:
$$
\pi_\calN (X^\alpha e_I) = \begin {cases}
X^\alpha e_I &\text{si $X^\alpha e_I \in \calN$} \\
0           &\text{sinon} \\
\end{cases}
$$

\begin {defn}

Pour un sous-module monomial $\calM$ de $\Mmac_{k-1}$, on définit l'endomorphisme
$B_\calM(\uP)$ de $\calM$ de la manière suivante:
$$
B_\calM(\uP) = \pi_\calM \circ B_k(\uP)_{|\calM} \overset{\rm def.} =
\pi_\calM \circ \partial_k(\uP) \circ \varphi_{|\calM}
$$
où l'on rappelle d'une part que $X^\beta
e_J\in \Mmac_{k-1} \Rightarrow X^\beta \in \Jex_1$ et d'autre part
que $\varphi : \Mmac_{k-1} \rightarrow \Smac_k$ est l'isomorphisme
monomial de référence donné par $\varphi(X^\beta e_J)
= \frac{X^\beta}{X_i^{d_i}} e_i \wedge e_J$ avec $i
= \minDiv(X^\beta)$.  Et comme $i < J$
par définition de $\Mmac_{k-1}$, on obtient un résultat
analogue à celui donné pour $B_{k,d}(\uP)$ dans la
proposition \ref{ExpressionBkXbetaeJ}:
$$
B_\calM(\uP)(X^\beta e_J)
\ = \ 
\pi_{\calM} 
\Bigg( 
\dfrac{X^\beta}{X_i^{d_i}}\  P_i \ e_J 
\quad + \quad 
\sum_{j \in J} 
\pm\, \dfrac{X^\beta}{X_i^{d_i}} \ 
P_j \ e_{i \vee J \setminus j}
\Bigg)
\qquad 
\text{où $i = \minDiv(X^\beta)$}
$$
\end {defn}

\index{B@les endomorphismes!3@$B_\calM(\uP)$ de $\calM$, monomial de Macaulay}%

\medskip

Au lieu de considérer le jeu circulaire, cela ne coûte pas plus cher
de le remplacer par le jeu circulaire généralisé intervenant dans le chapitre
\ref{ChapJeuCirculaire}, désigné par le même symbole $\uQ$ et
défini au dessus de $\bbZ[p_1, \dots, p_n, q_1, \dots, q_{n}]$ 
par 
$$
Q_i \ =\ p_i X_i^{d_i} \, + \, q_i X_i^{d_i-1}X_{i+1}
$$ 
Pour $i \in \DivSeq(X^\alpha)$:
$$
\dfrac{X^\alpha}{X_i^{d_i}} \, Q_i \ = \ 
p_i\, X^\alpha \, + \, 
q_i\,\dfrac{X^\alpha}{X_i} \,X_{i+1} 
$$
De manière à pouvoir utiliser le cadre de la
définition~\ref{DecompositionTriangulaireDef}, on a
besoin de spécifier des structures d'ordre sur les $X^\beta e_J$
avec $\#J = k-1$. Tout d'abord, on fixe 
un ordre monomial sur $\bfA[\uX]$ tel que $X_1 > \cdots > X_n$.

Concernant les parties $J$ de cardinal $k-1$ de $\{1..n\}$, on
assimile $J$ à la suite croissante de ses éléments, ce qui permet de
définir deux relations d'ordre. Une première d'ordre partiel, notée
$J' \preceq J$, a la signification que $J'$ est inférieure \emph
{terme à terme} à la suite $J$.  Par exemple,
$\{1,3,4\} \prec \{2,3,5\}$.  La seconde est la relation d'ordre
(total) lexicographique notée $J' \preceq_{\rm lex} J$. Il est clair que
$J' \preceq J \Rightarrow J' \preceq_{\rm lex} J$.

Et on décrète que:
$$
X^{\beta'}e_{J'} \leqslant \ X^\beta e_J 
\quad \text{si} \quad
\left\{
\begin{array} {l}
J' \prec J \text{ où $\prec$ est la relation d'ordre partiel ci-dessus}
\\
\text {ou bien} \\
J' = J \text{ et $X^{\beta'} \leqslant X^\beta$  au sens de l'ordre monomial fixé}
\end{array}
\right.
$$
Cette relation d'ordre partiel est couverte par la relation d'ordre total
$X^{\beta'}e_{J'} \leqslant_{\rm lex} \ X^\beta e_J$ définie par:
$$
X^{\beta'}e_{J'} \leqslant_{\rm lex} \ X^\beta e_J 
\quad \text{si} \quad
\left\{
\begin{array} {l}
J' \prec_{\rm lex} J
\\
\text {ou bien} \\
J' = J \text{ et $X^{\beta'} \leqslant X^{\beta}$}
\end{array}
\right.
$$
On peut alors énoncer l'analogue du résultat \ref{WkdQTriangSup} consacré à
$W_\calM$ pour $\calM \subset \Jex_{2,d}$.

\begin{prop}\label{BkdQTriangSup}
Soit un sous-module monomial $\calM \subset \Mmac_{k-1,d}$ avec $k \geqslant 2$.

\noindent
L'endomorphisme $B_{\calM}(\uQ)$ de $\calM$ est triangulaire relativement à l'ordre $\leqslant$
au sens suivant:
$$
B_{\calM}(\uQ)(X^\beta e_J) 
\ \in \ 
\bigoplus_{X^{\beta'} e_{J'} \leqslant X^\beta e_J} \bfA X^{\beta'} e_{J'}
$$
En conséquence, dans la base monomiale de $\calM$ rangée de manière
croissante pour l'ordre $\leqslant_{\rm lex}$, sa matrice est
triangulaire supérieure avec des $p_i$ sur la diagonale.  A fortiori,
$\det B_\calM(\uQ)$ est un produit de $p_i$.

\noindent
En particulier, pour le jeu circulaire $Q_i = X_i^{d_i} -
X_i^{d_i-1}X_{i+1}$, on a $\det B_{k,d}(\uQ) = 1$ pour tout $d$.
\end{prop}

\begin{proof}

Soit $X^\beta e_J \in \Mmac_{k-1}$. En posant $i = \minDiv(X^\beta)$
et $X^{\beta'} = \dfrac{X^\beta}{X_i}X_{i+1}$:
$$
\begin {array}{rclll}
B_\calM(\uQ)(X^\beta e_J) &=& \pi_\calM \Big(\dfrac{X^\beta}{X_i^{d_i}}Q_i e_J\Big) &+&
\text{combinaison $\bfA[X]$-linéaire de $e_{i \vee J \setminus j}$ où $j \in J$}
\\
&=& p_i\,X^\beta e_J \ + \ q_i\,\pi_\calM (X^{\beta'}e_J) &+& 
\text{combinaison $\bfA[X]$-linéaire de $e_{i \vee J \setminus j}$ où $j \in J$}
\\
\end {array}
$$
$\rhd$
En notant $J' = i \vee J \setminus j$, on a $J' \prec J$ comme le
montre le schéma ci-dessous où $\ell$ est l'indice de $j$ dans~$J$
\idest{} $j = j_\ell$:
$$
\begin{array} {*{19}c}
J : && j_1 & < & j_2 & < & \cdots & < &
   j_{\ell-1} & < & j_{\ell} & < & j_{\ell+1} & < & \cdots & < & j_{k-2} & < & j_{k-1}
\\ [0.5em]
J' : && i & < & j_1 & < & \cdots & < & 
   j_{\ell-2} & < & j_{\ell-1} & < & j_{\ell+1} & < &\cdots & < & j_{k-2} & < & j_{k-1}
\\
\end{array}
$$
$\rhd$
Comparons $X^{\beta'}$ et $X^\beta$. On a $i < \min J$ et comme $J$
est non vide (puisque $\#J = k-1$ et $k \ge 2$), on~a $i < n$ donc
$X_{i+1} < X_i$; en multipliant par $\frac{X^\beta}{X_i}$, on obtient
$X^{\beta'} < X^\beta$.

\medskip
On a donc obtenu, par définition de la relation d'ordre $\leqslant$ sur les $X^\beta e_J$
$$
B_{\calM}(\uQ)(X^\beta e_J) 
\ \in \ 
\bigoplus_{X^{\beta'} e_{J'} \leqslant X^\beta e_J} \bfA X^{\beta'} e_{J'}
$$
ce qui termine la preuve.
\end{proof}

\subsubsection*{Exemple avec $D = (1,1,2)$, $k=2$ et $d=4$}

On considère  l'ordre lexicographique sur $\bfA[X_1,X_2,X_3]$ tel que $X_1 > X_2 > X_3$
et on range la base monomiale de $\Mmac_{1,4}$ de manière croissante pour l'ordre
$\leqslant_{\rm lex}$ sur les $X^\beta\,e_J$:
$$
B_{2,4}(\uQ) \ = \ 
\EastBordermatrix{
p_{1} & . & . & . & . & . & \VR -q_{3} & . & . & . & . & \Heti{X_{1}X_{3}^{2}\,e_{2}} \\ 
. & p_{1} & . & q_{1} & . & . & \VR . & -q_{3} & . & . & . & \Heti{X_{1}X_{2}X_{3}\,e_{2}} \\ 
. & . & p_{1} & . & q_{1} & . & \VR . & . & . & . & . & \Heti{X_{1}X_{2}^{2}\,e_{2}} \\ 
. & . & . & p_{1} & . & . & \VR . & . & . & . & . & \Heti{X_{1}^{2}X_{3}\,e_{2}} \\ 
. & . & . & . & p_{1} & q_{1} & \VR . & . & . & . & . & \Heti{X_{1}^{2}X_{2}\,e_{2}} \\ 
. & . & . & . & . & p_{1} & \VR . & . & . & . & . & \Heti{X_{1}^{3}\,e_{2}} \\ 
\HR{11} 
. & . & . & . & . & . & \VR p_{2} & q_{2} & q_{1} & . & . & \Heti{X_{2}X_{3}\,e_{3}} \\ 
. & . & . & . & . & . & \VR . & p_{2} & . & q_{1} & . & \Heti{X_{2}^{2}\,e_{3}} \\ 
. & . & . & . & . & . & \VR . & . & p_{1} & . & . & \Heti{X_{1}X_{3}\,e_{3}} \\ 
. & . & . & . & . & . & \VR . & . & . & p_{1} & q_{1} & \Heti{X_{1}X_{2}\,e_{3}} \\ 
. & . & . & . & . & . & \VR . & . & . & . & p_{1} & \Heti{X_{1}^{2}\,e_{3}} \\ 
}
$$

\bigskip

Puisque $\det B_{k,d} = \Delta_{k+1,d}\,\Delta_{k,d}$ et $\Delta_{n+1,d} = 1$, on en
déduit, pour le jeu circulaire $\uQ$, que $\Delta_{k,d}(\uQ) = 1$ pour $k \ge 2$
et pour tout~$d$. L'application du théorème~\ref{PsatReg} fournit le corollaire
suivant.

\begin {coro} \label{DeltakdPsatreg}
Soit $\uP$ le système générique.
Pour $k \ge 2$ et pour tout $d$, les scalaires $\det B_{k,d}(\uP)$ et $\Delta_{k,d}(\uP)$
sont réguliers et réguliers sur $\bfA[\uX]/\uPsat$.
\end {coro}

\subsection {La forme $r_{0,d}$-linéaire alternée canonique
$\Det_{d,\protect\uP}:\bfA[\protect\uX]_d^{\calS_{0,d}} \to \bfA$ de $\rmK_{\sbullet,d}(\protect\uP)$}
\label{CanonicalCayleyDetSection}


Nous devons apporter une précision importante concernant le statut \og
à un inversible près\fg{} de ``la'' forme $\mu_d :
\BW^{r_{0,d}}(\bfA[\uX]_d) \to \bfA$.
Ou encore si l'on veut de ``la'' forme $r_{0,d}$-linéaire alternée
$$
\mu_d : \bfA[\uX]_d^{r_{0,d}} = \bfA[\uX]_d \times\cdots\times \bfA[\uX]_d \to \bfA
$$
Pour l'instant, nous sommes obligés de prendre des pincettes dans
cette écriture ou plutôt ne jamais oublier qu'une telle forme $\mu_d$
nécessite le choix d'un système d'orientations de
$\rmK_{\sbullet,d}(\uP)$ et que deux choix différents introduisent un
inversible.  Procédant ainsi, nous ne disposons pas d'une
définition canonique sauf dans le cas particulier où $r_{0,d} \le 1$
(voir la suite).

\medskip
\centerline{
\fbox{\parbox{0.95\linewidth}{
Nous allons voir que tout s'arrange si on remplace
l'entier~$r_{0,d}$ par la base monomiale $\calS_{0,d}$ de~$\Smac_{0,d}$
dont $r_{0,d}$ est le cardinal.}
}}

\medskip


C'est essentiellement pour cette raison que nous avons tenu à utiliser
ici $r_{0,d}$ au lieu du symbole~$\chi_d$ parfois utilisé auparavant (qui lui est égal
mais évocateur de la caractéristique d'Euler-Poincaré).
Nous n'oublions pas non plus (cf.~\ref{PassageQuotientFormeAlternee})
que cette forme $\mu_d$ passe au quotient modulo $\uP$ pour induire
une forme $r_{0,d}$-linéaire alternée sur $\bfB_d =
\bfA[\uX]_d/\langle\uP\rangle_d$. Mais ici, nous avons choisi de
rester au niveau $\bfA[\uX]_d$.

\medskip
Remplacer $r_{0,d}$ par $\calS_{0,d}$ signifie qu'il faut voir
$\bfA[\uX]_d^{r_{0,d}}$ comme le $\bfA$-module libre de rang~$r_{0,d}$
constitué des familles indexées par la base monomiale $\calS_{0,d}$ de
$\Smac_{0,d}$, ou encore des fonctions $\calS_{0,d} \to \bfA[\uX]_d$,
module que nous notons~$\AXdSod$.  En quelque sorte,
on a remplacé l'entier~$r_{0,d}$ par sa \og vraie nature\fg{} qui est
la base monomiale de $\Smac_{0,d}$. Et nous allons oeuvrer pour
obtenir une forme $r_{0,d}$-linéaire alternée \emph{canonique}:
$$
\Det_d : \AXdSod \to \bfA
$$
qui bien entendu \og passe au quotient modulo $\uP$ \fg{}
pour induire une forme $\Det_d : \bfB_d^{\calS_{0,d}} \to \bfA$.

\medskip

Le $\bfA$-module $\AXdSod$ contient un
élément privilégié noté~$\iota$: c'est l'injection canonique 
de $\calS_{0,d}$ dans~$\bfA[\uX]_d$. Une fois la forme
$\Det_d$ mise en place (c'est l'objet de cette section), on
pourra l'évaluer en~$\iota$ pour obtenir l'égalité correspondant à
$\mu_d(\bfe_{\Smac_{0,d}}) = \Delta_{1,d}$
(cf. le point ii) de la proposition~\ref{SuiteDeltak}):

$$
\Det_{d,\uP}(\iota) = \Delta_{1,d}(\uP) = \Delta^\sigma_{1,d}(\uP) \qquad
\text{pour n'importe quelle $\sigma \in \fS_n$}
$$

\subsubsection*{Base versus base ordonnée}

Les décompositions du complexe $\rmK_{\sbullet,d}(\uX^D)$ qui nous intéressent sont celles
coïncidant en degré homologique $0$ avec la décomposition
$\rmK_{0,d} \overset{\rm def}{=} \bfA[\uX]_d = \Jex_{1,d} \oplus \Smac_{0,d}$.
C'est le cas de la décomposition de Macaulay ordinaire et de celles
tordues par $\sigma \in \fS_n$.

\emph {Fixons} une telle décomposition
$(M_{\sbullet,d}, S_{\sbullet,d}, \varphi)$ et montrons que l'on peut définir une
forme $r_{0,d}$-linéaire alternée précise $\AXdSod
\to \bfA$, précise étant en opposition avec ``à un inversible près''.
Nous constaterons ensuite qu'elle ne dépend pas de la décomposition choisie
et nous pourrons alors décréter qu'elle est \emph {canonique}.

\medskip

Nous allons utiliser le fait que le $\bfA$-module libre $\Smac_{0,d}$
de dimension $r_{0,d}$ possède une base canonique, à savoir sa base
monomiale $\calS_{0,d}$.  Mais celle-ci n'est pas ordonnée sauf si bien sûr
$r_{0,d} \le 1$ (c'est le cas par exemple si $d \ge \delta$).
C'est cette nuance subtile entre ``base'' et ``base
ordonnée'' qui fait toute la différence entre 
$\AXdSod$ et  $\bfA[\uX]_d^{r_{0,d}}$
et va nécessiter de prendre quelques précautions.

\label {NOTA14-S0d}%

\medskip

Nous nous limitons dans la suite aux systèmes d'orientations
de la décomposition choisie de~$\rmK_{\bullet,d}(\uX^D)$ ayant la particularité suivante:
l'orientation $\bfe_{\Smac_{0,d}}$ est \emph {monomiale} i.e. 
$\bfe_{\Smac_{0,d}}$ est le produit extérieur, dans un certain ordre, de
la base monomiale $\calS_{0,d}$:
$$
\bfe_{\Smac_{0,d}} = m_1 \wedge m_2 \wedge \cdots  \qquad\qquad
\calS_{0,d} = \{m_1, m_2, \dots\}
$$
Une numérotation $\{m_1, m_2, \cdots\}$ de $\calS_{0,d}$
met en bijection $\calS_{0,d}$ et l'ensemble d'entiers $\{1..r_{0,d}\}$.
Elle permet donc de définir un isomorphisme:
$$
\theta : \AXdSod \overset{\simeq}{\longrightarrow}
\bfA[\uX]_d^{r_{0,d}}
$$
que l'on peut décrire de la manière suivante avec des notations
fonctionnelles: pour $f \in \AXdSod$ \idest{} une
fonction $f : \calS_{0,d} \to \bfA[\uX]_d$,
l'image $\theta(f)$ est la fonction $\{1..r_{0,d}\} \to \bfA[\uX]_d$
définie par $i \mapsto f(m_i)$.

Considérons maintenant deux systèmes d'orientations
de la décomposition choisie de $\rmK_{\bullet,d}(\uX^D)$ induisant:
$$
\bfe_{\Smac_{0,d}} = m_1 \wedge m_2 \wedge \cdots, \qquad\qquad
\bfe'_{\Smac_{0,d}} = m'_1 \wedge m'_2 \wedge \cdots
$$
Chacun des deux systèmes donne naissance à un système d'orientations
$\bfe_{M_{\sbullet,d}} \wedge \bfe_{S_{\sbullet,d}}$ de
$\rmK_{\sbullet,d}(\uX^D)$, par conséquent fournit une forme
$r_{0,d}$-linéaire alternée $\bfA[\uX]_d^{r_{0,d}} \to \bfA$.  Notons
$\mu_d$ et $\mu'_d$ les deux formes obtenues.  Ce que nous affirmons, c'est que $\mu_d
\circ\theta = \mu'_d\circ\theta'$: c'est la forme $\Det_d = \Det_{d,\uP}$
convoitée. Nous l'énonçons en utilisant un diagramme pour clarifier.

\begin {theo} \leavevmode

Dans le contexte ci-dessus, avec des notations implicites pour le système
$\uP$, on a l'égalité encadrée:  
$$
\xymatrix @R=1.5cm @C=2.5cm{
\AXdSod\ar[dr]|{\boxed{\mu_d\circ\theta=\mu'_d\circ\theta'}}
   \ar@/^8pt/[r]^{\theta}|{\simeq} \ar@/_8pt/[r]_{\theta'}|{\simeq}
   &\bfA[\uX]_d^{r_{0,d}}\ar@/_5pt/[d]_{\mu} \ar@/^5pt/[d]^{\mu'}  \\
   & \bfA \\
}   
$$
De plus, la forme $r_{0,d}$-linéaire alternée $\mu_d\circ\theta : \AXdSod
\to \bfA$ est indépendante de la décomposition du complexe $\rmK_{\sbullet,d}(\uX^D)$ pourvu
que cette décomposition induise $\bfA[\uX]_d = \Jex_{1,d}\oplus \Smac_{0,d}$
en degré homologique~0. Nous donnons le nom $\Det_d$ à cette forme:
$$
\Det_d = \Det_{d,\uP} : \AXdSod \to \bfA
$$
En désignant par $\iota \in \AXdSod$ l'inclusion $\calS_{0,d} \hookrightarrow
\bfA[\uX]_d$, elle vérifie:
$$
\Det_d(\iota) = \Delta_{1,d}
$$
\end {theo}

\label {NOTA14-Detd}%
%
%

\begin {proof} \leavevmode

$\rhd$  
Soit $\tau$ la permutation de l'ensemble $\{1..r_{0,d}\}$ telle que
$m'_i = m_{\tau(i)}$ de sorte que $\bfe'_{\Smac_{0,d}} =
\varepsilon(\tau)\, \bfe_{\Smac_{0,d}}$.  D'après le point ii) du
lemme \ref{MultiplicateursOrientationsDecomposition}:
$$
\mu'_d = \varepsilon(\tau)^{-1}\,\mu_d
$$
Alors pour $f \in \AXdSod$:
$$
\begin {array} {ccl}
(\mu'_d \circ \theta')(f) &=& \mu'_d\big(f(m'_1), f(m'_2), \dots \big) =
                        \mu'_d\big(f(m_{\tau(1)}), f(m_{\tau(2)}), \dots \big) \\
           &=& \varepsilon(\tau)\, \mu'_d\big(f(m_1), f(m_2), \dots \big) \\
           &=& \varepsilon(\tau)\, \varepsilon(\tau)^{-1} \mu_d\big(f(m_1), f(m_2), \dots \big) 
           = (\mu_d \circ\theta)(f) \\                     
\end {array}
$$

\medskip
$\rhd$
L'indépendance vis-à-vis de la décomposition provient essentiellement
du fait que, pour un système d'orientations $\bfe := (\bfe_0, \dots, \bfe_n)$
de~$\rmK_{\sbullet,d}(\uX^D)$ \emph{fixé}, il existe, pour n'importe quelle
décomposition de~$\rmK_{\sbullet,d}(\uX^D)$, un système d'orientations
de cette décomposition adapté à $\bfe$, cf. le
lemme \ref{ExistenceSystemeAdapte}.  Ainsi, en désignant par $\mu_d :
\bfA[\uX]_d^{r_{0,d}} \to \bfA$ la forme $r_{0,d}$-linéaire alternée
définie par $\bfe$, toutes ces décompositions $\bfe$-orientées
ont, \fbox{au niveau $\bfA[\uX]_d^{r_{0,d}} \to \bfA$}, même forme $\mu_d$.

Pour conclure, nous pouvons supposer~$\bfe_0$ de la forme
$\bfe_{\Jex_{1,d}} \wedge \bfe_{\Smac_{0,d}}$ avec
$\bfe_{\Smac_{0,d}}$ monomiale i.e. produit extérieur de la base
monomiale $\calS_{0,d}$ énumérée dans un certain ordre.  Cette
numérotation fournit un isomorphisme totalement indépendant de toute
décomposition de $\rmK_{\sbullet,d}(\uX^D)$:
$$
\theta : \AXdSod \overset{\simeq}{\longrightarrow}
\bfA[\uX]_d^{r_{0,d}}
$$
Ainsi \fbox{au niveau $\AXdSod \to \bfA$},
$\mu_d\circ\theta$ est la forme $r_{0,d}$-linéaire alternée de toute décomposition $\bfe$-orientée.

\medskip
$\rhd$
Fixons une décomposition $\bfe$-orientée, notons $\mu_d$ sa forme, $\{m_1, m_2, \dots\}$ une numérotation
de $\calS_{0,d}$ vérifiant $\bfe_{\Smac_{0,d}} = m_1 \wedge m_2 \wedge \cdots$ et
$\theta : \AXdSod \simeq \bfA[\uX]_d^{r_{0,d}}$ l'isomorphisme correspondant.
Alors:
$$
\Delta_d(\iota) \overset{\rm def}{=}
(\mu_d \circ \theta)(\iota) = \mu_d(m_1, m_2, \dots) =
\mu_d(\bfe_{\Smac_{0,d}}) = \Delta_{1,d}
$$
la dernière égalité provenant du point ii) de la proposition~\ref{SuiteDeltak}.

\end {proof}

\subsubsection*{Homogénéité de $\Det_{d,\uP}(\sbullet)$ en $P_i$ et le cas du jeu étalon généralisé}

$\rhd$
Puisque $\Det_{d,\uP}(\iota) = \Delta_{1,d}(\uP)$, la forme $\Det_{d,\uP}$ a le même degré d'homogénéité
en $P_i$ que le scalaire~$\Delta_{1,d}(\uP)$. Celui-ci sera déterminé plus loin (théorème~\ref{poidsDelta1}).
Cela signifie donc, pour des scalaires quelconques $a_1, \dots, a_n$, que l'on a:
$$
\forall\, f \in \AXdSod, \qquad \Det_{d,(a_1P_1, \dots, a_nP_n)}(f) = \prod_{i=1}^n a_i^{\dim \Jex_{1\setminus2,d}^{(i)}} \times
\Det_{d,(P_1, \dots, P_n)}(f)   
$$
où l'on rappelle que $\Jex_{1\setminus2,d}^{(i)}$ est le sous-module monomial
de $\bfA[\uX]_d$ de base les monômes ayant $i$ comme seul indice de divisibilité.

\medskip
$\rhd$
La forme $\Det_{d,\uX^D}$ est étroitement liée à la forme $\det
: \End_\bfA(\Smac_{0,d}) \to \bfA$.  Notons $\pi = \pi_{\Smac_{0,d}} :
\bfA[\uX]_d \to \Smac_{0,d}$ la projection parallèlement à
$\Jex_{1,d}$ et pour chaque $f \in \AXdSod$, notons
$\widetilde f : \Smac_{0,d} \to \bfA[\uX]_d$ son prolongement
$\bfA$-linéaire.  Alors $\pi \circ \widetilde f$ est un endomorphisme
de $\Smac_{0,d}$ et:
$$
\Det_{d,\uX^D}(f) = \det(\pi \circ \widetilde f) 
$$
De manière plus terre à terre, en notant $S_f$ la matrice strictement carrée
indexée par $\calS_{0,d}$:
$$
S_f = \big(\coeff_{X^\beta}f(X^\alpha)\big)_{X^\beta,X^\alpha \in \calS_{0,d}}
$$
on a
$$
\Det_{d,\uX^D}(f) = \det(S_f)
$$
On en déduit, pour le jeu étalon généralisé $\bsp\uX^D = (p_1X_1^{d_1}, \dots, p_nX_n^{d_n})$,
la formule close
$$
\Det_{d,\bsp\uX^D}(f) =\prod_{i=1}^n p_i^{\dim \Jex_{1\setminus2,d}^{(i)}} \times \det(S_f) 
$$

\subsubsection*{La détermination effective de la forme $\Det_{d,\uP}$}

Nous concrétisons ici la proposition \ref{muDelta2Expression}
concernant l'obtention de la forme alternée intervenant au numérateur de $\Det_d$:
$$
\Det_d = \frac{\text{$r_{0,d}$-forme linéaire alternée}}{\Delta_{2,d}}
\qquad \text{version intrinsèque de} \qquad
\mu_d = \frac{\big(\bigwedge^{r_{1,d}}(\Syl_d)(\bfe_{S_{1d}})\big)^\sharp}{\Delta_{2,d}}
$$
Ces égalités utilisent un certain nombre de notations implicites qu'il convient
de préciser. Il y a bien sûr le système $\uP$, mais nous avons l'habitude
d'écrire $\Delta_{2,d}$ pour $\Delta_{2,d}(\uP)$, $\Syl_d$ pour $\Syl_d(\uP)$,
etc. Plus subtil est l'utilisation implicite, que ce soit
aux numérateurs et aux dénominateurs, de la décomposition de Macaulay
$(\Mmac_{\sbullet,d}, \Smac_{\sbullet,d}, \varphi)$
du complexe $\rmK_{\sbullet,d}(\uX^D)$, décomposition
pilotée par $\minDiv$. Utiliser la version tordue par $\sigma \in \fS_n$
implique d'aménager en conséquence numérateurs et dénominateurs,
sans changer pour autant $\Det_d$. 

\medskip

Venons-en à $\Det_d$ versus $\mu_d$. Contrairement à $\mu_d$,
qui prend comme argument un $r_{0,d}$-uplet de polynômes homogènes
de~$\bfA[\uX]_d$ et qui nécessite des choix d'orientations, la forme
$\Det_d$ prend comme argument une fonction $f
: \calS_{0,d} \to \bfA[\uX]_d$ et n'a besoin de rien d'autre.
La forme $r_{0,d}$-linéaire alternée qui intervient au numérateur de l'expression
ci-dessus de $\Det_d$ est définie comme le déterminant d'un certain endomorphisme
de $\bfA[\uX]_d$ de la manière suivante.

\begin {defn}[La construction $\bsOmega_d : \AXdSod \to \End(\bfA[\uX\rbrack_d)$
pilotée par $\minDiv$] \leavevmode   
\label{OmegaDetNumerator}
    
Pour  $f \in \AXdSod$,  on construit un endomorphisme $\bsOmega_d(f) = \bsOmega_{d,\uP}(f)$ de~$\bfA[\uX]_d$
défini sur sa base monomiale par:
$$
\forall\, |\alpha| = d, \qquad
\bsOmega_d(f) :  X^\alpha \mapsto \begin {cases}
f(X^\alpha)                   &\text{si $X^\alpha \in \Smac_{0,d}$} \\[3mm]
\dfrac{X^\alpha}{X_i^{d_i}}P_i &\text{si $X^\alpha \in \Jex_{1,d}$ où $i = \minDiv(X^\alpha)$} \\
\end {cases}
$$
Pour $\sigma \in \fS_n$, en remplaçant $\minDiv$ par $\sminDiv$, on obtient
une version dite tordue par $\sigma$, que l'on note $\bsOmega^\sigma_d$.
Et la proposition \ref{muDelta2Expression} s'énonce ainsi:
$$
\Det_d(f) = \frac{\det\big(\bsOmega_d(f)\big)}{\Delta_{2,d}} =
\frac{\det\big(\bsOmega^\sigma_d(f)\big)}{\Delta^\sigma_{2,d}}
\qquad \forall\ \sigma \in \fS_n
$$
Bien entendu, c'est une manière symbolique d'écrire les égalités sans dénominateurs:
$$
\det\big(\bsOmega_d(\sbullet)\big) = \Delta_{2,d} \Det_d(\sbullet), \qquad
\det\big(\bsOmega^\sigma_d(\sbullet)\big) = \Delta^\sigma_{2,d} \Det_d(\sbullet)
$$
\end {defn}

\label {NOTA14-Omegad}%
\index{O@Omega-constructeur à valeurs dans!$\End_\bfA(\bfA[\uX]_d)$!$\bsOmega_d=\bsOmega_{d,\uP}$,\quad$\bsOmega^\sigma_d=\bsOmega^\sigma_{d,\uP}$}%

\begin {rmq} [Le cas particulier $d \ge \delta$] \leavevmode

Pour $d \ge \delta$, on a $r_{0,d} \le 1$ et la base monomiale du
$\bfA$-module $\Smac_{0,d}$, libre de dimension $r_{0,d}$, se trouve
être ordonnée.

\medskip
$\rhd$ Cas $d \ge \delta+1$.
On a $\calS_{0,d}=\emptyset$ donc $\AXdSod$ est de
cardinal 1.  Ainsi $\bsOmega_d : \bfA[\uX]_d^0 \to \End(\bfA[\uX]_d)$
s'identifie à la valeur prise en l'unique élément de $\AXdSod$, valeur
qui est l'endomorphisme~$W_{1,d}$ de $\Jex_{1,d}$ défini
en~\ref{DefW1}, en cohérence avec le fait que 
$\bfA[\uX]_d = \Jex_{1,d}$.
A fortiori $\det\bsOmega_d(\iota) = \det W_{1,d}$.

\medskip
$\rhd$ Cas $d=\delta$.
On a $\calS_{0,\delta}=\{X^\emouton\}$ de sorte que
$\bfA[\uX]_\delta^{\calS_{0,\delta}}$ s'identifie à
$\bfA[\uX]_\delta$ et $\iota$ à $X^\emouton$.  Ainsi $\bsOmega_\delta
: \bfA[\uX]_\delta \to \End(\bfA[\uX]_\delta)$ n'est autre que la
construction~$\Omega$ de la section~\ref{DefEndoOmega}
et $\det\bsOmega_\delta(\iota) = \det W_{1,\delta}$.

\medskip

Dans les deux cas, en évaluant $\det\bsOmega_d(\sbullet) = \Delta_{2,d} \Det_d(\sbullet)$ en $\iota$,
on retrouve l'égalité:
$$
\det W_{1,d} = \Delta_{2,d} \Delta_{1,d}
$$
Précisons que l'égalité $\det\bsOmega_d(\iota) =\det W_{1,d}$ est
vérifiée pour tout $d$.  Elle n'est pas difficile mais nous préférons
la reporter en~\ref {DetOmegaiota=DetW1d} après les exemples et
surtout après les deux dessins évocateurs de la page~\pageref
{DessinOmegad}. Précisons également que cette égalité et celles
qui précèdent possèdent une version tordue par $\sigma \in \fS_n$.

Dans tous les exemples suivants qui vont nous servir
à illustrer $\bsOmega_d$ et $\Det_d$, nous prendrons donc $d < \delta$
\og pour y voir quelque chose\fg{}, le cas $d \ge \delta$ venant
d'être réglé par cette remarque et par ailleurs largement traité dans
les chapitres antérieurs.
\end {rmq}

\medskip

Voici un petit résultat général, certes élémentaire, mais bien utile pour
la production d'exemples (tels ceux montrés ci-après).

\begin {lem} [Détermination de $\Smac_{0,\delta-1}$ et de sa dimension $r_{0,\delta-1}$]
\leavevmode

La base monomiale de $\Smac_{0,\delta-1}$ est indexée par l'ensemble $I$ des indices
$i$ tels que $d_i > 1$. Elle est constituée des $(X^\emouton/X_i)_{i \in I}$ et en conséquence, $r_{0,\delta-1} = \#I$.  
\end {lem}
  
\begin {proof}

On rappelle que pour chaque $0 \le d \le \delta$, en posant $d' = \delta -d$, on a
$\Smac_{0,d} \simeq \Smac_{0,d'}$  via l'involution ``mouton-swap'' $X^\alpha \mapsto X^{D-\Un-\alpha}$.
En prenant $d = \delta-1$, on obtient $\Smac_{0,\delta-1} \simeq \Smac_{0,1}$. La
base monomiale de~$\Smac_{0,1}$ est constituée des $(X_i)_{i \in I}$, d'où le résultat.
\end {proof}

\subsubsection*{L'exemple du format $D = (1,2,3)$ de degré critique $\delta=3$,
de mouton-noir $YZ^2$ et de $d=2$}

Le système $\uP$ de format $D$ est le suivant:
$$
P_1 = a_1X_1 + a_2X_2 + a_3X_3,
\qquad\quad
P_2 = b_1X_1^2 + b_2X_1X_2 + b_3X_1X_3 + b_4X_2^2 + b_5X_2X_3 + b_6X_3^2
$$
Il est inutile de \og voir\fg{} $P_3$ car il n'est pas visualisé dans la matrice ci-dessous.
Nous faisons le choix de $d=2 < \delta=3$. Pour ce $d$, on a $r_{1,d}\overset{\rm def}{=}
\dim \Jex_{1,d} = 4$ et $r_{0,d} \overset{\rm def}{=}\dim\Smac_{0,d} = 2$ d'où
$$
r_{1,d}+ r_{0,d} = \dim \bfA[\uX]_d = \binom {d+{n-1}}{n-1} = \binom {d+2}{2}
\overset{\rm ici}{=} 6
$$

$\bullet$
Commençons par visualiser la construction $f \mapsto \bsOmega_d(f)$ via la matrice
strictement carrée ci-dessous:
$$
\overbrace{X^2,\ XY,\ XZ,\ Y^2}^{\Jex_{1,2}}\ ,\ 
\overbrace{YZ,\, Z^2}^{\Smac_{0,2}}
\qquad \qquad
\bsOmega_{d,\uP}(f) = 
\NorthEastBordermatrix{
\Veti{X\,e_{1}} & \Veti{Y\,e_{1}} & \Veti{Z\,e_{1}} & \Veti{e_{2}} &
\Veti{YZ} & \Veti{Z^2} & \\
a_{1} & . & . & b_{1}      &f(YZ)_{X^2}  &\vdots    & \Heti{X^{2}} \\
a_{2} & a_{1} & . & b_{2}  &f(YZ)_{XY}   &          & \Heti{XY} \\
a_{3} & . & a_{1} & b_{3}  &\vdots      &           & \Heti{XZ} \\
. & a_{2} & . & b_{4}      &            &\vdots     & \Heti{Y^{2}} \\
. & a_{3} & a_{2} & b_{5}  &            &f(Z^2)_{YZ} & \Heti{YZ} \\
. & . & a_{3} & b_{6}      &\vdots     &f(Z^2)_{Z^2} & \Heti{Z^{2}} \\
}
$$
Précisons que $f(YZ)$ et $f(Z^2)$ sont $r_{0,d}=2$ polynômes homogènes de degré
$d=2$ et que l'indice indique la coordonnée sur la base monomiale de $\bfA[X,Y,Z]_2$.
Pour des raisons typographiques, nous n'avons pas voulu surcharger
l'indexation des 2 colonnes-argument de droite.

$\bullet$ Enfin, comme $d < \min_{i\ne j}(d_i+d_j) = 3$, on a $\Delta_{k,d} = 1$
pour $k \ge 2$, a fortiori $\Delta_{2,d} = 1$ au dénominateur.  En voici
la justification:
$$
d < \min_{i\ne j}(d_i+d_j)  \iff (\forall k \ge 2,\ \Jex_{k,d} = 0) \iff
(\forall k\ge 2,\ \rmK_{k,d} = 0)
$$
Soit $k \ge 2$. A fortiori, on a $\Smac_{k,d} = 0$, ce qui équivaut à $\Mmac_{k-1,d} = 0$
donc $\det B_{k,d} = 1$. En utilisant $\det B_{k,d} = \Delta_{k,d} \Delta_{k+1,d}$, il vient
de proche en proche $\Delta_{k,d} = 1$.

\subsubsection*{L'exemple $D = (1,1,4)$ de degré critique $\delta=3$,
 de mouton-noir $Z^3$ et de $d=2 < \delta$}


Ici, $d = \min_{i \ne j}(d_i + d_j)$, ce qui va permettre d'illustrer la détermination
des dénominateurs $\Delta_{2,d}$ et~$\Delta^\sigma_{2,d}$ pour $\sigma = (1,2,3)$.
Et, en passant, le fait que l'on peut avoir $r_{0,d} = 1$ pour $0 < d < \delta$.
Voici les bases monomiales de $\Mmac^\sigma_{k,d}$, $\Smac^\sigma_{k,d}$ pour $k=0,1,2$
($k=3$ n'intervient pas car, de manière générale, $\rmK_{n,d} = 0$ pour $d < \sum_i d_i$): 
$$
\Mmac_{0,2} = \Jex_{1,2} : (X^2, XY, XZ, Y^2, YZ), \qquad
\Smac_{0,2} : Z^2
$$
Pour $k=1$, dans $\rmK_{1,2} = \bfA[X,Y,Z]_1\,e_1 \oplus \bfA[X,Y,Z]_1\,e_2$,
en rappelant que $\Mmac^\sigma_1$ est défini par les contraintes $X_i^{d_i} e_j$ avec $i <_\sigma j$.
$$
\begin {array}{ccc}
\Mmac_{1,2} : Xe_2,       &\qquad \Smac_{1,2} : (Xe_1, Ye_1, Ze_1, Ye_2, Ze_2) \\ [2mm]
\Mmac_{1,2}^\sigma : Ye_1, &\qquad \Smac^\sigma_{1,2} : (Xe_1, Xe_2, Ze_1, Ye_2, Ze_2) \\
\end{array}
$$
Et enfin pour $k=2$:
$$
\Mmac_{2,2} =\Mmac^\sigma_{2,2} = 0, \qquad \Smac_{2,2} = \Smac^\sigma_{2,2} = \bfA e_1\wedge e_2
$$
La forme $\Det_2$ est indépendante du polynôme $P_3$. On note:
$$
P_1 = a_1X + a_2Y + a_3Z, \qquad  P_2 = b_1X + b_2Y + b_3Z
$$
Interviennent dans $\Det_2$, les numérateurs suivants avec une seule colonne-argument
($r_{0,2} = 1$):
$$
\bsOmega_2 =
\NorthEastBordermatrix{
\Veti{X\,e_{1}} & \Veti{Y\,e_{1}} & \Veti{Z\,e_{1}} & \Veti{Y\,e_{2}} & \Veti{Z\,e_{2}} & \Veti{Z^2} & \\
a_{1} & . & . & . & .           &v_{11} & \Heti{X^{2}} \\
a_{2} & a_{1} & . & b_{1} & .    &v_{12} & \Heti{XY} \\
a_{3} & . & a_{1} & . & b_{1}    &v_{13} & \Heti{XZ} \\
. & a_{2} & . & b_{2} & .        &v_{22} & \Heti{Y^{2}} \\
. & a_{3} & a_{2} & b_{3} & b_{2} &v_{23}& \Heti{YZ} \\
. & . & a_{3} & . & b_{3}        &v_{33} & \Heti{Z^{2}} \\
}
\qquad\qquad
\bsOmega^\sigma_2 =
\NorthEastBordermatrix{
\Veti{X\,e_{1}} & \Veti{X\,e_{2}} & \Veti{Z\,e_{1}} & \Veti{Y\,e_{2}} & \Veti{Z\,e_{2}} & \Veti{Z^2} & \\
a_{1} & b_{1} & . & . & .        &v_{11} & \Heti{X^{2}} \\
a_{2} & b_{2} & . & b_{1} & .    &v_{12} & \Heti{XY} \\
a_{3} & b_{3} & a_{1} & . & b_{1} &v_{13} & \Heti{XZ} \\
. & . & . & b_{2} & .            &v_{22} & \Heti{Y^{2}} \\
. & . & a_{2} & b_{3} & b_{2}     &v_{23} & \Heti{YZ} \\
. & . & a_{3} & . & b_{3}         &v_{33} & \Heti{Z^{2}} \\
}
$$
En ce qui concerne les dénominateurs $\Delta_{2,2}^\sigma$, les matrices
$B^\sigma_{2,2}$ sont $1 \times 1$:
$$
B_{2,2} = \big[\coeff_{X^2}(P_1)\big], \qquad
B^\sigma_{2,2} = \big[\coeff_{Y^2}(P_2)\big]
$$  
D'où la détermination de la forme linéaire $\Det_2 : \bfA[X,Y,Z]_2\to \bfA$:
$$
\Det_2 = \frac{\det\bsOmega_2}{\Delta_{2,2}} =
\frac{\det\bsOmega^\sigma_2}{\Delta^\sigma_{2,2}},  \qquad
\Delta_{2,2} = a_1,\qquad \Delta^\sigma_{2,2} = b_2, \qquad
\sigma = (1,2,3)
$$
En notant $\uxi = (\xi_1,\xi_1,\xi_3)$ le vecteur des mineurs signés des 2 polynômes linéaires
$P_1,P_2$ \idest{} défini par:
$$
\begin {vmatrix}
a_1 & b_1 & T_1 \\
a_2 & b_2 & T_2 \\
a_3 & b_3 & T_3 \\
\end  {vmatrix}
= \xi_1 T_1 + \xi_2 T_2 + \xi_3 T_3
$$
le lecteur pourra vérifier que la forme linéaire $\Det_2$ s'identifie à l'évaluation en $\uxi$:
$$
\Det_2(v) = \sum_{1 \le i \le j \le 3} v_{ij}\xi_i \xi_j
$$
En fait, c'est un cas particulier du cas plus général $D = (1,\dots,1,e)$ et $d \le \delta = e-1$.
On a $\Smac_{0,d} = \bfA X_n^d$ donc $r_{0,d} = 1$. Ainsi $\AXdSod$ s'identifie à $\bfA[\uX]_d$
et $\Det_d : \bfA[\uX]_d \to \bfA$ est une forme linéaire: c'est l'évaluation en $\uxi$
où $\uxi = (\xi_1, \dots, \xi_n)$ est le vecteur des mineurs signés des $n-1$ premiers
polynômes linéaires, cf la preuve en~\ref{DetdCasEcoleI}.

\subsubsection*{L'exemple $D = (2,2,3)$ de degré critique $\delta=4$ et $d=3 < \delta$}


Soit $\uP$ le système de format $D = (2,2,3)$ de degré critique $\delta =4$, de mouton-noir $XYZ^2$:
$$
\setlength{\tabcolsep}{2pt}
\left\{
\begin{tabular}{rcp{15cm}} 
$P_{1}$ & $=$ & $a_{1}X^{2} + a_{2}XY + a_{3}XZ + a_{4}Y^{2} + a_{5}YZ + a_{6}Z^{2}$\\ [0.1cm] 
$P_{2}$ & $=$ & $b_{1}X^{2} + b_{2}XY + b_{3}XZ + b_{4}Y^{2} + b_{5}YZ + b_{6}Z^{2}$\\ [0.1cm] 
$P_{3}$ & $=$ & $c_{1}X^{3} + c_{2}X^{2}Y + c_{3}X^{2}Z + c_{4}XY^{2} + c_{5}XYZ + c_{6}XZ^{2} + c_{7}Y^{3} + c_{8}Y^{2}Z + c_{9}YZ^{2} + c_{10}Z^{3}$\\ [0.1cm] 
\end{tabular}
\right.
$$
Dans cet exemple, puisque $d = 3 < \min_{i\ne j} (d_i+d_j) = 4$, on a $\Delta_{2,d}(\uP) = 1$.
Voici les 10 monômes de degré~$3$ avec la visualisation du fait que $r_{0,3}=3$:
$$
\overbrace{X^3,\ X^2Y,\ X^2Z,\ XY^2,\ Y^3,\ Y^2Z,\ Z^3}^{\Jex_{1,3}}\ ,\ 
\overbrace{XYZ,\, XZ^2,\, YZ^2}^{\Smac_{0,3}}
$$
Il n'y a pas de dénominateur \idest{} $\Det_3 = \det\bsOmega_3$ où $\bsOmega_3$
est la matrice $10 \times 10$ suivante:
$$
\bsOmega_{3,\uP} = \
\NorthEastBordermatrix{
\Veti{X\,e_{1}} & \Veti{Y\,e_{1}} & \Veti{Z\,e_{1}} & \Veti{X\,e_{2}} & \Veti{Y\,e_{2}} & \Veti{Z\,e_{2}} & \Veti{e_{3}} &
\Veti{XYZ\leftrightarrow Z} & \Veti{XZ^2\leftrightarrow Y} & \Veti{YZ^2\leftrightarrow X} & \\
a_{1} & . & . & b_{1} & . & . & c_{1} &\sbullet &\sbullet &\sbullet & \Heti{X^{3}} \\
a_{2} & a_{1} & . & b_{2} & b_{1} & . & c_{2} &   &   &   & \Heti{X^{2}Y} \\
a_{3} & . & a_{1} & b_{3} & . & b_{1} & c_{3} &   &   &   & \Heti{X^{2}Z} \\
a_{4} & a_{2} & . & b_{4} & b_{2} & . & c_{4} &   &   &   & \Heti{XY^{2}} \\
. & a_{4} & . & . & b_{4} & . & c_{7} &   &   &   & \Heti{Y^{3}} \\
. & a_{5} & a_{4} & . & b_{5} & b_{4} & c_{8} &   &   &   & \Heti{Y^{2}Z} \\
. & . & a_{6} & . & . & b_{6} & c_{10} &   &   &   & \Heti{Z^{3}} \\
a_{5} & a_{3} & a_{2} & b_{5} & b_{3} & b_{2} & c_{5} &   &   &   & \Heti{XYZ} \\
a_{6} & . & a_{3} & b_{6} & . & b_{3} & c_{6} &   &   &   & \Heti{XZ^{2}} \\
. & a_{6} & a_{5} & . & b_{6} & b_{5} & c_{9} &\sbullet &\sbullet &\sbullet & \Heti{YZ^{2}} \\
}
\qquad
\begin {array}{c}
\left[
\begin {smallmatrix}
a_{1} & . & . & b_{1} & . & .               & c_{1} &. &. &. \\
a_{2} & a_{1} & . & b_{2} & b_{1} & .        & c_{2} &. &. &.   \\
a_{3} & . & a_{1} & b_{3} & . & b_{1}        & c_{3} &. &. &.   \\
a_{4} & a_{2} & . & b_{4} & b_{2} & .        & c_{4} &. &. &.   \\
. & a_{4} & . & . & b_{4} & .               & c_{7} &. &. &.   \\
. & a_{5} & a_{4} & . & b_{5} & b_{4}        & c_{8} &. &. &. \\
. & . & a_{6} & . & . & b_{6}               & c_{10}&. &. &.   \\
a_{5} & a_{3} & a_{2} & b_{5} & b_{3} & b_{2} & c_{5} &1 &. &. \\
a_{6} & . & a_{3} & b_{6} & . & b_{3}        & c_{6} &. &1 &.  \\
. & a_{6} & a_{5} & . & b_{6} & b_{5}        & c_{9} &. &. &1\\
\end{smallmatrix}
\right]
\\ [15mm]
\text{visualiser $\bsOmega_3(\iota)$} \\
\end {array}
$$
Cette matrice est strictement carrée et la correspondance entre
lignes/colonnes repose sur la correspondance biunivoque $\Smac_{1,d}
\leftrightarrow \Jex_{1,d}$, pilotée par $\minDiv$, fournie dans un
sens par $X^\beta e_i \mapsto X^\beta X_i^{d_i}$ et dans l'autre sens
par $X^\alpha \mapsto (X^\alpha/X_i^{d_i})\, e_i$ avec $i =
\minDiv(X^\alpha)$.

\medskip

Les 3 dernières colonnes, indexées par la base monomiale $\calS_{0,3}$
de $\Smac_{0,3}$, sont les \og colonnes-joker\fg.  En présence d'une
fonction $f : \calS_{0,3} \to \bfA[\uX]_3$, en y reportant dans
l'ordre indiqué les coordonnées des 3 polynômes $f(XYZ), f(XZ^2),
f(YZ^2)$ de $\bfA[X,Y,Z]_3$, on a $\Det_3(f) = \det \bsOmega_3(f)$.
Ceci achève la description de la forme $3$-linéaire alternée $\Det_3$.

\bigskip

Signalons une propriété supplémentaire, non vérifiée par les exemples
précédents, qui fera l'objet du chapitre ultérieur~\ref{ChapSylvesterHybride}.
Il s'agit d'une évaluation spéciale de $\Det_3$ qui fournit $\Res(\uP)$.

\medskip

En notant $d' = \delta-d \overset{\rm ici}{=} 1$, l'involution
mouton-swap $\calS_{0,d} \simeq \calS_{0,d'}$ réalise $X^\alpha
\leftrightarrow X^{\alpha'} := X^{D-\Un-\alpha}$ (le produit de
$X^\alpha$ et $X^{\alpha'}$ est le mouton-noir). Nous l'avons
schématisée dans les 3 dernières colonnes sous la forme
$XYZ\leftrightarrow Z$, $XZ^2 \leftrightarrow Y$ et
$YZ^2\leftrightarrow X$: pour $m \leftrightarrow m'$, le
produit $mm'$ est le mouton-noir $XYZ^2$.

On verra qu'il y a une fonction ``déterminant bezoutien''
$\bsnabla(\uP) : \calS_{0,d'} \to \bfA[\uX]_d$ qui associe à un monôme
$X^\beta$ de degré $d'$ un polynôme homogène $\bsnabla_\beta(\uP)$ de degré $d$.
On montrera, en composant cette fonction par $\calS_{0,d}
\overset{\simeq}{\longrightarrow} \calS_{0,d'}$, que $f :
\calS_{0,d} \to \bfA[\uX]_d$ définie par $X^\alpha \mapsto \bsnabla_{D-\Un-\alpha}(\uP)$
vérifie:
$$
\Det_{d,\uP}(f) = \Res(\uP)
$$
Assimilons $\bsnabla(\uP)$ à la famille
$\big(\bsnabla_\beta(\uP)\big)_{|\beta|=d'}$. Ici, on a donc une
famille $\big(\bsnabla_Z(\uP), \bsnabla_Y(\uP), \bsnabla_X(\uP)\big)$
qui vérifie: 
$$
\Res(\uP) = \det \bsOmega_{3,\uP}(\bsnabla_Z,\ \bsnabla_Y,\ \bsnabla_X)
$$
Cette égalité constitue donc une formule déterminantale via un
déterminant d'ordre 10 . A comparer
avec la formule $\Res(\uP) = \omegares(\nabla)$ qui fait intervenir le
quotient de 2 déterminants, d'ordre $\dim \bfA[\uX]_\delta = 15$ au
numérateur et d'ordre $\dim \Jex_{2,\delta} = 1$ au dénominateur
($\Jex_{2,\delta} = \bfA X^2Y^2$).
Quant à la formule $\Res(\uP) = \det W_{1,\delta+1}/\det
W_{2,\delta+1}$, le déterminant est d'ordre 21 au numérateur et
d'ordre 5 au dénominateur.

Une petite vérification sur le poids ne peut pas faire de mal. On a
$\widehat d_1 + \widehat d_2 + \widehat d_3 = 6+6+4 = 16$ et dans
l'ordre inverse des formules évoquées:
$$
16 = 21-5 = (15-1) + 3 - 1 = (10-3) + (3+3+3)
$$
Au milieu, le $-1$ dans $15-1$ correspond à la colonne-argument, $3$
au poids du déterminant bezoutien~$\nabla(\uP)$ et $-1$ au poids du
dénominateur.  Tandis qu'à droite, $-3$ dans $10-3$ correspond aux 3
colonnes-argument; il faut savoir que chaque déterminant bezoutien est
de poids $n=3$ en les coefficients des $P_i$, ce qui explique le
$3+3+3$ correspondant aux poids de $\bsnabla_Z,\ \bsnabla_Y,\ \bsnabla_X$.
Note: il y a beaucoup de 3 ici, qu'il ne faut pas confondre
$$
n=3\quad (\text{poids des bezoutiens}), \qquad d=3, \qquad r_{0,d} = 3
$$
Nous verrons dans le chapitre~\ref{ChapSylvesterHybride} que
cette propriété supplémentaire/particulière est due au fait que~$d$ appartient
à la fourchette:
$$
\delta - \min(D) < d \le \delta  \qquad \text{ce qui équivaut à} \qquad
0 \le d' < \min(D)
$$
La contrainte sur $d'$ a comme conséquence que $\Jex_{1,d'} = 0$, d'où
l'égalité $\Smac_{0,d'} = \bfA[\uX]_{d'}$.

\subsubsection*{A propos d'un glissement notationel et du risque de notations cryptiques}

Nous avons rencontré les objets suivants:
$$
\BW^{r_{0,d}}(\bfA[\uX]_d), \qquad
\bfA[\uX]_d^{r_{0,d}}, \qquad
\AXdSod
$$
Bien entendu, nous avons une flèche du second vers le premier qui
consiste à associer à un $r_{0,d}$-uplet $(F_1, F_2, \dots)$ de
polynômes homogènes de degré $d$ leur produit extérieur $F_1 \wedge
F_2 \wedge \cdots$ dans $\bigwedge^{r_{0,d}}(\bfA[\uX]_d)$ et qui
permet de considérer comme équivalente la donnée d'une forme
$r_{0,d}$-linéaire alternée sur $\bfA[\uX]_d$ et celle d'une forme
linéaire sur $\BW^{r_{0,d}}(\bfA[\uX]_d)$.

Pour des raisons de normalisation de signes, nous avons dû opérer avec
l'objet de droite plutôt qu'avec celui du milieu, remplaçant un
$r_{0,d}$-uplet par une fonction $f : \calS_{0,d} \to \bfA[\uX]_d$
qu'il est parfois préférable de voir comme une famille
$\big(f(X^\alpha)\big)_{X^\alpha \in \calS_{0,d}}$ de $\bfA[\uX]_d$.
Mais, en présence d'une telle $f \in \AXdSod$, rien n'empêche de
penser à une espèce de produit extérieur (mal défini lorsque
$\calS_{0,d}$ n'est pas ordonné)
$$
\bigwedge_{X^\alpha\in\calS_{0,d}} \kern -5pt f(X^\alpha) \in \BW^{r_{0,d}}(\bfA[\uX]_d)
$$
Nous essaierons dans la suite de déjouer les écritures trop cryptiques.
Dans le schéma ci-dessous, on a $i = \minDiv(X^\alpha)$. Pour la version
tordue par $\sigma \in \fS_n$, il faut remplacer $\bsOmega_d$ par
$\bsOmega^\sigma_d$, $i = \minDiv(X^\alpha)$ par $i = \sminDiv(X^\alpha)$
et $\Smac_{1,d}$ par $\Smac^\sigma_{1,d}$.

\centerline{
\begin{tikzpicture}[scale = 0.8]
\draw[fill=gray!20]  (0,0) rectangle (4,6) ;
\draw[fill=gray!60]  (4,0) rectangle (6,6) ;
\path (0,6) -- (4,6) node[midway, above] {$\Jex_{1,d} \simeq \Smac_{1,d}$} ;
\path (4,6) -- (6,6) node[midway, above] {$\Smac_{0,d}$} ;
\path (6,0) -- (6,6) node[midway,right] {$\bfA[\uX]_d$} ;
\path (0,0) -- (4,6) node[midway] {$\Syl_d(\uP)_{|\Smac_{1,d}}$} ;
\path (4,0) -- (6,6) node[midway] {$f$} ;
\path (0,0) -- (0, 6) node[midway, left] {$\bsOmega_d(f) \ : \ \quad$} ;
\end{tikzpicture}
\qquad
\begin{tikzpicture}[scale = 0.8]
\draw [thick] (-0.1, -0.3) -- (-0.4, -0.3) -- (-0.4, 6.3) -- (-0.1, 6.3) ;
\draw [thick] (6+0.1, -0.3) -- (6+0.4, -0.3) -- (6+0.4, 6.3) -- (6+0.1, 6.3) ;
\draw [dashed, rounded corners=4pt, fill=gray!20] (1.8, -0.3) rectangle (2.2, 6.3) ;
\draw (2,6.3) node[above] {\footnotesize $\dfrac{X^\alpha}{X_i^{d_i}} P_i$} ;
\draw [dotted] (2,4) -- (6.6,4) node[right] {\footnotesize $X^\alpha \in \Jex_{1,d}$} ;
\draw [dashed, rounded corners=4pt, fill=gray!60] (4.8,-0.3) rectangle (5.2,6.3) ;
\draw (5,6.3) node[above] {\footnotesize $f(X^\gamma)$} ;
\draw [dotted] (5,1) -- (6.6,1) node[right] {\footnotesize $X^{\gamma} \in \Smac_{0,d}$} ;
\foreach \r in {0,1,...,6} \draw (\r, 6-\r) node {\tiny $\bullet$} ;
\end{tikzpicture}
}

\label{DessinOmegad}

\bigskip

Rappelons que l'on a désigné par $\iota : \calS_{0,d} \hookrightarrow \bfA[\uX]_d$ 
l'injection canonique.  La proposition suivante relie l'endomorphisme
$\bsOmega_d(\iota)$ de $\bfA[\uX]_d$ et l'endomorphisme $W_{1,d}$ de
$\Jex_{1,d}$.

\begin {prop}
\label {DetOmegaiota=DetW1d}
Ces deux endomorphismes ont même déterminant:
$$
\det \bsOmega_d(\iota) = \det W_{1,d}
$$  
Idem pour la version tordue par une permutation $\sigma \in \fS_n$.
\end {prop}

\begin {proof}
Ecrivons $\bfA[\uX]_d = \Jex_{1,d} \oplus \Smac_{0,d}$. Dans la base monomiale de $\bfA[\uX]_d$,
réunion de celle de $\Jex_{1,d}$ et de celle de $\Smac_{0,d}$, l'endomorphisme $\bsOmega_d(\iota)$ a
pour matrice:
$$
\bsOmega_d(\iota) = \begin {bmatrix}
W_{1,d} & 0  \\
*      & \Id_{\Smac_{0,d}} \\
\end {bmatrix}
$$
d'où le résultat.
\end {proof}

\subsubsection*{Les propriétés de type \og Cramer\fg{} des formes $\det\bsOmega_{d,\uP}$ et $\Det_{d,\uP}$}

\begin {defn}
Pour un système $\uP$ et $f \in \AXdSod$, on note:
$$
\bsomega_{d,\uP}(f) = \det \bsOmega_{d,\uP}(f)
\qquad \text{ou, en omettant $\uP$,} \qquad
\bsomega_d(f) = \det \bsOmega_d(f)
$$
Notation analogue $\bsomega^\sigma_d(f) = \det \bsOmega^\sigma_d(f)$ pour $\sigma \in \fS_n$
de sorte que:
$$
\bsomega^\sigma_d(\sbullet) = \Delta^\sigma_{2,d}\, \Det_d(\sbullet)
$$
En évaluant ces formes en $\iota$, on retrouve l'égalité:
$$
\det W^\sigma_{1,d} = \Delta^\sigma_{2,d}\,\Delta_{1,d}
$$
\end {defn}

\medskip

Pour $F,G \in \bfA[\uX]_\delta$, on a souvent utilisé la propriété
$\omega(F)G - \omega(G)F \in \langle\uP\rangle_\delta$ avec comme
conséquence le fait que $\omega(F)G \in \langle\uP\rangle_\delta +
\bfA F$. Et en particulier, l'implication $F \in \uPsat_\delta \Rightarrow
\omega(F) \in \ElimIdeal$ (prendre pour $G$ les $X^\alpha$ avec
$|\alpha|= \delta$ ou au choix les $X_i^{\delta}$). Voici un premier
résultat, sorte de pendant pour $\bsomega_d$.

\label {NOTA14-bsomega}%
%
%

\begin {lem}   
Soit $f \in \AXdSod$.

\begin{enumerate}[\rm i)]
\item    
Pour $G \in \bfA[\uX]_d$, on a:
$$
\bsomega_{d,\uP}(f)\,G  \ \in\ \langle\uP\rangle_d + \sum_{|\alpha| = d} \bfA f(X^\alpha) 
$$  
\item
En particulier, si \emph{chaque} $f(X^\alpha) \in \uPsat_d$, alors $\bsomega_d(f) \in \ElimIdeal$.

\item
Si $f(X^\alpha) \in \langle \uP\rangle_d$ pour \emph{un} $X^\alpha$, alors $\bsomega_d(f) = 0$. Dit autrement,
$\bsomega_d$ passe au quotient modulo $\uP$, induisant une forme $r_{0,d}$-linéaire
alternée $\bfB_d^{\calS_{0,d}} \to \bfA$.
\end {enumerate}  
\end {lem}

\begin {proof} \leavevmode

i)  
Par définition, $\bsomega_{d,\uP}(f)$ est un déterminant d'ordre $m := \dim \bfA[\uX]_d$,
que nous écrivons avec des crochets $[Q_1, \dots, Q_m]$, chaque $Q_j \in \bfA[\uX]_d$
étant, ou bien un multiple d'un $P_i$ ou bien un $f(X^\alpha)$. En utilisant la
formule de Cramer (voir le début du chapitre \ref{ChapAC1}), on a:
$$
[Q_1, \dots, Q_m]\,G \ =\ [G,Q_2,\dots, Q_m]\,Q_1 + [Q_1,G,\dots, Q_m]\,Q_2 + \cdots +
[Q_1,Q_2,\dots, G]\,Q_m  
$$  
d'où l'appartenance annoncée.

\medskip
ii)
En prenant $G = X_i^d$, on obtient  $\bsomega_d(f)\,X_i^{d} \in \uPsat$ pour tout $i$ donc
$\bsomega_d(f) \in \ElimIdeal$.

\medskip
iii) De manière générale, dans le déterminant $[Q_1, \cdots, Q_m]$, il
y a (au moins) $r_{1,d}$ indices $j$ tel que $Q_j$ soit multiple d'un
$P_i$, où l'on rappelle que $r_{1,d} \overset{\rm def}{=} \dim
\Jex_{1,d}$.  Si de plus, il y a un $X^\alpha$ tel que
$f(X^\alpha) \in \langle\uP\rangle_d$, on dispose de $r_{1,d}+1$ indices $j$ tels que
$Q_j \in \langle \uP\rangle_d$. D'où la nullité annoncée puisque
$\Syl_d(\uP)$ est de rang~$\le r_{1,d}$.
\end {proof}

\medskip

Remarquons que l'on aurait pu déduire le point iii) de l'égalité
$$
\bsomega_{d,\uP}(f) = \Delta_{2,d}(\uP)\, \Det_{d,\uP}(f)
$$
et du fait que $\Det_{d,\uP}(f)$ passe au quotient modulo $\uP$. Nous
avons préféré le démontrer directement pour insister sur l'importance
du fait que~$\Syl_d(\uP)$ est de rang $\le r_{1,d}$.

\medskip
Avouons que ce lemme n'est qu'une version provisoire: il ne fournit
pas le pendant pour $\bsomega_d$ de l'appartenance $\omega(F)G -
\omega(G)F \in \langle\uP\rangle_\delta$. Nous l'avons donné car il
évitait les notations cryptiques que nous devons affronter maintenant.
Supposons par exemple que $r_{0,d}=2$ si bien que $\bsomega_d$ est une forme
$2$-linéaire alternée sur $\bfA[\uX]_d$. Nous avons l'intention d'écrire
quelque chose du style
$$
\forall\, F_1,F_2,G \in \bfA[\uX]_d, \quad 
\bsomega_d(F_1,F_2)G \ \equiv\ \bsomega_d(G,F_2)F_1 + \bsomega_d(F_1,G)F_2
\quad \bmod \langle\uP\rangle_d
$$
Pour y parvenir, nous devons réaliser une certaine gymnastique, à
savoir pour $f \in \AXdSod$, $X^\alpha \in \calS_{0,d}$ et $G
\in \bfA[\uX]_d$, définir une fonction $\fSwap{\alpha}{G}$ consistant à
remplacer la valeur de $f(X^\alpha)$ par $G$. Voici sa définition:
$$
\begin {array}{cccl}
\fSwap{\alpha}{G} : &\calS_{0,d}& \to &\bfA[\uX]_d  \\[2mm]
             &X^{\alpha'}& \longmapsto &
\begin {cases} f(X^{\alpha'}) &\text{si $\alpha'\ne\alpha$}\\ G&\text{si $\alpha'=\alpha$} \\ \end {cases}  
\end {array}
$$
Cette définition cryptique permet maintenant d'énoncer la proposition suivante dont la seule
difficulté réside dans les notations. En clair, nous n'avons pas à fournir de preuve.

\begin {prop} [Propriété \og Cramer\fg{} pour $\bsomega_d$]
Dans le contexte ci-dessus, on dispose de la congruence:
$$
\bsomega_d(f)G \ \equiv \sum_{X^\alpha \in \calS_{0,d}}
\bsomega_d(\fSwap{\alpha}{G})\, f(X^\alpha) \quad \bmod \langle\uP\rangle_d
$$
\end {prop}

Voilà maintenant le résultat important qui mérite le nom de théorème
(et une preuve!): nous allons remplacer $\bsomega_d$ par $\Det_d$,
pendant du passage, en degré $\delta$, de $\omega$ à $\omegares$.

\index{propriété Cramer!3@de la forme $\chi_d$-linéaire alternée $\Det_d$}%

\begin {theo} [Propriété \og Cramer\fg{} pour $\Det_d$]
\label{DetdIsCramer}
En conservant les notations précédentes:
$$
\Det_{d,\uP}(f)\,G \ \equiv \sum_{X^\alpha \in \calS_{0,d}}
\Det_{d,\uP}(\fSwap{\alpha}{G})\,f(X^\alpha) \quad \bmod \langle\uP\rangle_d
$$
En conséquence, si $f(X^\alpha) \in \uPsat_d$ pour chaque $X^\alpha \in \calS_{0,d}$,
alors $\Det_{d,\uP}(f) \in \ElimIdeal$.
\end {theo}  

\begin {proof}
Pour démontrer le théorème, nous pouvons supposer $\uP$ générique. Reprenons l'égalité:
$$
\Delta_{2,d}(\uP)\, \Det_{d,\uP}(f) = \bsomega_{d,\uP}(f)
$$
et notons $H \in \bfA[\uX]_d$ la différence des deux membres de l'énoncé.
D'après la proposition précédente, en travaillant dans $\bfB_d =
\bfA[\uX]_d/\langle\uP\rangle_d$:
$$
\Delta_{2,d}(\uP)\, H = 0  \qquad \text{dans $\bfB_d$}
$$
Nous n'avons plus qu'à justifier le fait que $\Delta_{2,d}(\uP)$ est
$\bfB_d$-régulier.  Il résulte du corollaire~\ref{DeltakdPsatreg}
qu'il est régulier sur $\bfA[\uX]/\uPsat$ et la remarque qui suit
le corollaire~\ref{corBdeltaReg} assure qu'il est $\bfB_d$-régulier.

\medskip
En ce qui concerne la deuxième assertion, prendre pour $G$ les $X_i^d$ pour $1 \le i \le n$.
\end {proof}

\subsection{Homogénéité/poids en $P_i$ des différentielles de Koszul
     $\partial_{k,d}(\protect\uP)$ et des $\det B_{k,d}(\protect\uP)$}
\label{PoidsEnPiBkd}

Ici la suite $\uP$ est générique de sorte que $\bfA = \bfk[\indetsPi]$.
L'objectif est l'étude des degrés d'homogénéité en $P_i$ des mineurs de
$\partial_{k,d}(\uP) : \rmK_{k,d} \to \rmK_{k-1,d}$, en particulier
celle des déterminants des endomorphismes~$B_{k,d}(\uP)$.  Il s'agit
donc de prolonger les résultats de la proposition~\ref{PoidsDetW}
consacrée au cas $k = 1$ et au poids des $\det W_\calM(\uP)$,
mineurs extraits de $\Syl_d(\uP) = \partial_{1,d}(\uP)$.

Le cas $k$ quelconque est un peu plus compliqué  que le cas $k=1$
et nécessite de (pouvoir) graduer convenablement les deux $\bfA$-modules libres
de rang fini $\rmK_{k,d}$ et $\rmK_{k-1,d}$, en cohérence avec la
structure graduée de l'anneau de polynômes $\bfA$, et de manière à ce que
$\partial_{k,d}(\uP)$ soit graduée de poids~$0$.
La graduation de chaque $\rmK_{k,d}$ va être mise en place par l'intermédiaire
de sa base monomiale:
$$
\rmK_{k,d} = \bigoplus_{\substack{|\alpha| + d_I=d\\\#I=k}} \bfA X^\alpha e_I
\qquad \qquad
d_I = \sum_{i \in I} d_i
$$

\medskip
On peut considérer que ce qui suit repose sur le lemme
général ci-dessous dont la preuve élémentaire est laissée au
lecteur. A noter que dans cet énoncé, $a_{t,s}$ ou $\det(A)$ peut
être nul tout en possédant un poids spécifié.

\begin{lem} \label{PoidsDeterminant}
Soit une matrice carrée $A = (a_{t,s})_{t \in T, \,s \in S}$ 
($S$~pour Source, $T$ pour Target)
à coefficients dans un anneau gradué par un monoïde commutatif $\Delta$
noté additivement, comme $\bbN$, $\bbN^n$ ou~$\bbZ^n$.
On suppose disposer de $p_S : S \rightarrow \Delta$ 
et $p_T : T \rightarrow \Delta$ telles que chaque
$a_{t,s}$ soit homogène de poids $p_S(s) - p_T(t)$.
Alors $\det(A)$, défini au signe près selon la manière d'ordonner $S$,
$T$, est homogène de poids $\sum_{s \in S} p_S(s) - 
\sum_{t \in T} p_T(t)$. 
\end{lem}

\medskip

Il va être commode de munir $\bfA = \bfk[\indetsPi]$ d'une structure $\bbN^n$-graduée en
convenant de $\poids(\text{n'importe quel coefficient de $P_i$}) = \varepsilon_i$, où $\varepsilon_i$
désigne le $i$-ème vecteur de la base canonique de~$\bbN^n$. Ainsi, dire 
qu'un élément $a$ de $\bfA$ (donc un polynôme en les coefficients des $P_i$) est homogène de poids
$m_1\varepsilon_1 + \cdots + m_n\varepsilon_n$, où les $m_i$ sont dans $\bbN$, signifie
tout simplement que $a$ est homogène en $P_i$ de poids~$m_i$.

\smallskip

On attribue alors à la $\bfA$-base monomiale $(X^\alpha\,e_I)$ de $\rmK_\sbullet$ le poids suivant
$$
\poids(X^\alpha\,e_I) = \sum_{i\in I} \varepsilon_i
$$
Le bien-fondé de ce choix, qui n'est pas arbitraire, repose
sur le lemme suivant et sa preuve.  On conseille au lecteur d'examiner
tout d'abord l'exemple qui suit le lemme. La preuve de celui-ci est
reportée après l'exemple.

\index{poids en $P_i$ et $\bbN^n$-graduation de $\rmK_\sbullet(\uP)$}%

\begin {lem} \label{PoidsMineurKoszul}  
\leavevmode

\begin {enumerate}[\rm i)]
\item
L'application $\partial_{k,d}(\uP) : \rmK_{k,d} \to \rmK_{k-1,d}$ est graduée
de poids~$0$.

\item  
Soient une partie $S$ de la base monomiale de $\rmK_{k,d}$ et
une partie $T$ de même cardinal de la base monomiale de $\rmK_{k-1,d}$.
Alors le mineur de $\partial_{k,d}(\uP)$ extrait sur $T \times S$
est homogène en les $P_i$ de poids 
$\poids(S) - \poids(T) \overset{\rm def.}{=} \sum_{s \in S} \poids(s) - \sum_{t \in T} \poids(t)$.
\end {enumerate}
\end {lem}

\begin {exemple}
Prenons $D = (2,1,1)$, $k=2$, $d=3$.  Le système $\uP$ est le suivant:
$$
\setlength{\tabcolsep}{2pt}
\left\{
\begin{tabular}{rcp{13cm}} 
$P_{1}$ & $=$ & $a_{1}X^{2} + a_{2}XY + a_{3}XZ + a_{4}Y^{2} + a_{5}YZ + a_{6}Z^{2}$\\ [0.1cm] 
$P_{2}$ & $=$ & $b_{1}X + b_{2}Y + b_{3}Z$\\ [0.1cm] 
$P_{3}$ & $=$ & $c_{1}X + c_{2}Y + c_{3}Z$\\ [0.1cm] 
\end{tabular}
\right.
$$
et $\partial_{2,d} : \rmK_{2,d} \to \rmK_{1,d}$ est donnée par:
$$
\partial_{2,3}(\uP) =
\NorthEastBordermatrix{
\Veti{e_{12}} & \Veti{e_{13}} & \Veti{X\,e_{23}} & \Veti{Y\,e_{23}} & \Veti{Z\,e_{23}} &\\
-b_{1} & -c_{1} & . & . & . & \Heti{X\,e_{1}} \\
-b_{2} & -c_{2} & . & . & . & \Heti{Y\,e_{1}} \\
-b_{3} & -c_{3} & . & . & . & \Heti{Z\,e_{1}} \\
a_{1} & . & -c_{1} & . & . & \Heti{X^{2}\,e_{2}} \\
a_{2} & . & -c_{2} & -c_{1} & . & \Heti{XY\,e_{2}} \\
a_{3} & . & -c_{3} & . & -c_{1} & \Heti{XZ\,e_{2}} \\
a_{4} & . & . & -c_{2} & . & \Heti{Y^{2}\,e_{2}}  \\
a_{5} & . & . & -c_{3} & -c_{2} & \Heti{YZ\,e_{2}} \\
a_{6} & . & . & . & -c_{3} & \Heti{Z^{2}\,e_{2}} \\
. & a_{1} & b_{1} & . & . & \Heti{X^{2}\,e_{3}} \\
. & a_{2} & b_{2} & b_{1} & . & \Heti{XY\,e_{3}} \\
. & a_{3} & b_{3} & . & b_{1} & \Heti{XZ\,e_{3}} \\
. & a_{4} & . & b_{2} & . & \Heti{Y^{2}\,e_{3}} \\
. & a_{5} & . & b_{3} & b_{2} & \Heti{YZ\,e_{3}} \\
. & a_{6} & . & . & b_{3} & \Heti{Z^{2}\,e_{3}}  \\
}
\
\let\s=\scriptstyle
\begin {array}{l}
\s1\\\s2\\\s3\\\s4\\\s5\\\s6\\\s7\\\s8\\\s9\\\s10\\\s11\\\s12\\\s13\\\s14\\\s15\\ 
\end {array}
$$
Examinons la première colonne, vecteur du $\bfA$-module d'arrivée $\rmK_{1,d}$, libre de dimension 15
$$
-b_1 Xe_1 - b_2Ye_1 - \cdots + a_1X^2e_2 + a_2XYe_2 + \cdots
$$
Relativement à la graduation de $\rmK_{1,d}$, cette première colonne est homogène de poids
$\varepsilon_1 + \varepsilon_2$; par exemple, dans $-b_1Xe_1$, le coefficient $b_1$ de $P_2$, est homogène de poids
$\varepsilon_2$ et $e_1$ de poids $\varepsilon_1$. Les indéterminées $X,Y,Z$ ne contribuent à aucun
poids car on a décrété que $\poids(X^\alpha\,e_I) = \sum_{i\in I} \varepsilon_i$.

Ce poids d'homogénéité $\varepsilon_1 + \varepsilon_2$ de la
première colonne est celui du premier élément $e_{12}$ de la base de
l'espace $\rmK_{2,d}$ de départ. Idem pour les autres colonnes.
C'est en ce sens là que l'application $\partial_{2,d}(\uP)$ est graduée de poids $0$, car
elle transforme un vecteur homogène de l'espace de départ en un vecteur
homogène \textit{de même poids} de l'espace d'arrivée.

\medskip
\'Ecrire $\rmK_{2,d} \simeq \bfA^5$ est insuffisant pour rendre compte
des graduations:
$$
\rmK_{2,d} = \bfA\,e_{12} \oplus \bfA\,e_{13} \oplus \bfA\,Xe_{23} \oplus
\bfA\,Ye_{23} \oplus \bfA\,Ze_{23} 
$$
Il est classique d'écrire l'information graduée sous forme de twists de l'anneau gradué $\bfA$:
$$
\rmK_{2,d} \simeq \bfA(-\varepsilon_1-\varepsilon_2) \oplus \bfA(-\varepsilon_1-\varepsilon_3)
\oplus \big(\bfA(-\varepsilon_2-\varepsilon_3)\big)^3
$$
mais nous n'en dirons pas plus.

\medskip

Illustrons maintenant le deuxième point du lemme.
Considérons un exemple de mineurs $T \times S$
en fournissant la partie $S$ des colonnes et la partie $T$ des lignes, sous forme de parties ordonnées des indices :
$$
S = (1,2,3,4), \qquad T = (1,2,5,6)
$$
Autrement dit, on a 
$$
S = (e_{12}, e_{13}, X e_{23}, Y e_{23})
\qquad \text{ et } \qquad 
T = (Xe_1, Ye_1, XY e_2, XZe_2)
$$
Alors:
$$
\poids(S) =
\sum_{s \in S} \poids(s) = (\varepsilon_1 + \varepsilon_2) +
(\varepsilon_1 + \varepsilon_3) + (\varepsilon_2 + \varepsilon_3) + (\varepsilon_2 + \varepsilon_3) =
2\varepsilon_1 + 3\varepsilon_2 + 3\varepsilon_3
$$
tandis que $\poids(T) = 2\varepsilon_1 + 2\varepsilon_2$. On a
$\poids(S)-\poids(T) = \varepsilon_2 + 3\varepsilon_3$. Et on peut vérifier que
le mineur en question:
$$
\begin {vmatrix}
-b_{1} & -c_{1} & . & . \\
-b_{2} & -c_{2} & . & . \\
a_{2} & . & -c_{2} & -c_{1} \\
a_{3} & . & -c_{3} & . \\
\end {vmatrix}
\quad
\text {est bien de poids $0$ en $P_1$, de poids $1$ en $P_2$ et de poids $3$ en $P_3$}
$$
Un lecteur courageux pourra réaliser les calculs de mineurs ci-dessous dont
voici les poids:
$$
\let\vep=\varepsilon
\begin {array}{c|c|c|c|}
S,T                   &(1,2,3,4),(1,4,5,6) &(1,2,3,4),(3,11,12,14) &(2,3,4,5),(3,4,7,8)\\[1mm]
\hline
\vrule height12pt depth3pt width0pt
\poids(S)-\poids(T) &\vep_1+3\vep_3        &\vep_1+3\vep_2       &4\vep_3\\
\end {array}
$$
\end {exemple}


\begin {proof} [Preuve du lemme \ref{PoidsMineurKoszul}]
 \leavevmode

\medskip
i)
L'application $\partial_{k,d} : \rmK_{k,d} \rightarrow \rmK_{k-1, d}$ réalise
$X^\alpha e_I \mapsto \sum_{i \in I} \pm P_i X^\alpha\,e_{I \setminus i}$.
Par définition de la graduation, chaque terme de $P_i$ 
est homogène de poids $\varepsilon_i$ (ne pas oublier qu'un
monôme en $\uX$ de $P_i$ n'apporte aucune contribution au poids) ; donc
$\pm P_i X^\alpha\,e_{I \setminus i}$ est homogène de poids $\sum_{i\in I} \varepsilon_i$,
égal au poids de $X^\alpha e_I$.
Ainsi l'application $\partial_{k,d}$ est graduée de poids~$0$.

\smallskip
Examinons la matrice de $\partial_{k,d}$ dans les bases monomiales
rangées dans n'importe quel ordre.  Le coefficient en position
$(X^{\alpha'} e_{I'}, X^\alpha e_I) \in \rmK_{k-1,d} \times
\rmK_{k,d}$ est donné par :
$$
\text{coefficient de } \partial_{k,d}(X^\alpha e_I) \text{ sur } X^{\alpha'} e_{I'} 
\ = \ 
\left \{
\begin{array}{ll}
\pm \,\text{coeff($P_i$)} & \text{si $X^\alpha \mid X^{\alpha'}$
et $I' = I \setminus i$} \vspace{0.2cm}\\
0 & \text{sinon} 
\end{array}
\right.
$$
Nous affirmons que ce coefficient est homogène de poids égal à $\poids
(X^\alpha e_I)- \poids(X^{\alpha'} e_{I'})$.  En effet, ce coefficient
est nul à l'exception du cas où $I' = I \setminus i$ et $X^\alpha \mid
X^{\alpha'}$ pour lequel il vaut $\pm {\rm coeff}(P_i)$.  Dans les
deux cas, le coefficient est homogène de poids $\varepsilon_i$ qui est
bien égal au poids annoncé $\sum_{j \in I} \varepsilon_j - \sum_{j\in
  I'} \varepsilon_{j}$.

\medskip
ii) En ce qui concerne l'homogénéité et le poids du mineur $T \times S$, on applique
le résultat général du lemme~\ref{PoidsDeterminant} en prenant $p_S = p_T = \poids$.
\end {proof}

\begin{prop}\label{PoidsDetBkdSum}
Le déterminant $\det B_{k,d}$ est homogène de poids 
$$ 
p \ = \
\sum_{X^\beta e_J \, \in \, \Mmac_{k-1,d}} 
\!\!\!\!\! \varepsilon_{\minDiv(X^\beta)}
$$
Autrement dit, le poids en $P_i$ de $\det B_{k,d}$ est égal à $\dim \Mmac_{k-1,d}^{(i)}$ où
l'indice $(i)$ spécifie $\minDiv = i$:
$$
\poids_{P_i}(\det B_{k,d}) 
\ = \ 
\#\big\{X^\beta e_J \in \Mmac_{k-1,d} \ \big | \ \minDiv(X^\beta) = i \big\}
$$
On dispose également d'une seconde égalité:
$$
\poids_{P_i}(\det B_{k,d}) 
\ = \ 
\#\big\{ X^\alpha e_I \in \Smac_{k,d} \ \big | \ \min I = i \big\}
$$
En particulier, $\det B_{k,d}$ ne dépend pas des $k-1$ derniers polynômes~$P_i$.
\end{prop}

\begin{proof} 
Le fait que les différentielles de Koszul soient graduées de degré $0$ permet 
d'appliquer le résultat général~\ref{PoidsDeterminant} à 
$S = \{ X^\alpha e_I \in \Smac_{k,d} \}$, $T = \{ X^\beta e_J \in \Mmac_{k-1,d} \}$ 
et aux fonctions $p_S$ et $p_T$, 
restrictions de la fonction $\poids$.

On obtient alors que $\det B_{k,d}$ est homogène de poids $p$:
$$
p = \sum_{s \in S} \poids(s) - \sum_{t \in T} \poids(t)
$$
Faisons intervenir l'isomorphisme monomial $\varphi$, 
figurant dans la définition de l'endomorphisme $B_{k,d}$ : 
$$
\vcenter{
\xymatrix @M=0.4pc @R=0.2pc @C=5pc{
\Mmac_k & \Mmac_{k-1} \ar@<0.5ex>@{-->}[ddl]^{\varphi} \\
\oplus & \oplus \\
\Smac_k \ar@<0.5ex>[ruu]^-{\beta_k} & \Smac_{k-1} \\
}
}
\qquad \qquad
\varphi(X^\beta e_J) \ = \ \dfrac{X^\beta}{X_i^{d_i}} e_i \wedge e_J 
\quad \hbox{avec $i = \minDiv(X^\beta)$}
$$
On obtient alors 
$p = \sum_{t\in T} \big(\poids\big(\varphi(t)\big) - \poids(t)\big)$; or avec les
notations ci-dessus, on a
$$
\poids\big(\varphi(X^\beta e_J)\big) - \poids(X^\beta e_J) \ = \ 
\varepsilon_i 
\ = \ 
\varepsilon_{\minDiv(X^\beta)}
$$
D'où l'égalité donnée pour $p$. 
La seconde égalité résulte de la remarque suivante :
lorsque $X^\beta e_J \in \Mmac_{k-1,d}$ et $X^\alpha e_I \in \Smac_{k,d}$
se correspondent par $\varphi$, alors $\minDiv(X^\beta) = \min I$.

\medskip
En ce qui concerne la dernière assertion, supposons $\poids_{P_i}(\det B_{k,d}) \geqslant 1$
\idest{} il existe un ${X^\beta\,e_J \in \Mmac_{k-1}}$ tel que $i = \minDiv(X^\beta)$.
Comme $X^\beta\,e_J \in \Mmac_{k-1}$, on a $i < J$ et $\#J = k-1$ donc $i$ ne
figure pas parmi les $k-1$ derniers éléments de $\{1..n\}$.
\end{proof}

\medskip

On peut utiliser une stratégie analogue pour les
déterminants~$W_{h,d}$ et retrouver le résultat déjà obtenu dans la
proposition \ref{PoidsDetW}. Ne pas oublier que l'énoncé formel
suivant dit simplement que le déterminant de $W_{h,d}$ est homogène
en~$P_i$, de poids le nombre de colonnes de type~$P_i$.

\begin{prop}
Le déterminant $\det W_{h,d}$ est homogène de poids 
$$ 
p \ = \
\sum_{X^\alpha \, \in \, \Jex_{h,d}}
\!\! \varepsilon_{\minDiv(X^\alpha)} 
$$
Autrement dit, le poids en $P_i$ de $\det W_{h,d}$ vaut
$$
\# \big\{ 
X^\alpha \in \Jex_{h,d} \ \big | \ \minDiv(X^\alpha) = i
\big\}
\ = \ 
\dim \Jex_{h,d}^{(i)}
$$
En particulier, $\det W_{h,d}$ ne dépend pas des $h-1$ derniers polynômes.
Plus généralement, $\det W_{h,d}^\sigma$ ne dépend pas de $P_i$ où $i$ parcourt les $h-1$ 
derniers éléments de $\big(\sigma(1), \dots, \sigma(n)\big)$.
\end{prop}

\begin{proof} \leavevmode

On applique le résultat général~\ref{PoidsDeterminant} avec
$T = \{ X^\alpha  \in \Jex_{h,d} \}$, $S = \varphi(T)$ et
$\varphi(X^\alpha) = (X^\alpha/X_i^{d_i})\,e_i$ où $i = \minDiv(X^\alpha)$.
Puisque $\poids\big(\varphi(X^\alpha)\big) - \poids(X^\alpha) = \varepsilon_i$,
le poids $p$ de $\det W_{h,d}$ est donné par :
$$
p = \sum_{t \in T} \poids\big( \varphi(t) \big) - \sum_{t \in T} \poids(t) =
   \sum_{X^\alpha \, \in \, \Jex_{h,d}} \!\! \varepsilon_{\minDiv(X^\alpha)} 
$$
\end {proof}

\subsubsection*{A propos de la \og non dépendance en les derniers polynômes\fg{} des deux
familles de déterminants}

On a vu en~\ref{ProprieteWh} que~$W_{h,d}$ ne dépend pas des $h-1$
derniers polynômes~$P_i$, a fortiori, il en est de même de son
déterminant.  Quant à $\det B_{k,d}$, bien qu'il ne dépende pas des
$k-1$ derniers polynômes, la matrice $B_{k,d}$ peut, elle, en
dépendre.  En voici un petit exemple pour $D = (1,1,1)$, $P_1 = aX$,
$P_2 = b' Y$ et $P_3 = cX$. Alors $B_{2,2} = \left [
\begin{smallmatrix}
a & . & -c \\
. & a & . \\
. & . & b' \\
\end{smallmatrix}
\right ]$
dépend du troisième polynôme mais pas son déterminant.
Idem avec les polynômes génériques :
$$
P_1 = aX + a' Y + a'' Z, \qquad 
P_2 = bX + b' Y + b'' Z, \qquad 
P_3 = cX + c' Y + c'' Z
$$ 
conduisant à un endomorphisme légèrement plus compliqué 
$$
B_{2,2} \ = \ 
\EastBordermatrix{
a  & . & -c & \Heti{X e_2} \\ 
. & a  & b & \Heti{X e_3} \\ 
. & a' & b' & \Heti{Y e_3} \\ 
}
$$
On voit cependant aisément que  $\det B_{2,2}$ ne dépend pas du dernier polynôme.

Voici un exemple plus complexe pour $D = (3,1,2)$, $k=2$ et $d=5$ :
$$
B_{2,5} \ =  \ 
\EastBordermatrix{
a_{1} & . & . & . & -c_{1} & . & . & . & . & . & \Heti{X_{1}^{4}\,e_{2}} \\ 
a_{2} & a_{1} & . & . & -c_{2} & -c_{1} & . & . & . & . & \Heti{X_{1}^{3}X_{2}\,e_{2}} \\ 
a_{3} & . & a_{1} & . & -c_{3} & . & -c_{1} & . & . & . & \Heti{X_{1}^{3}X_{3}\,e_{2}} \\ 
. & . & . & a_{1} & b_{1} & . & . & . & . & . & \Heti{X_{1}^{3}\,e_{3}} \\ 
. & . & . & a_{2} & b_{2} & b_{1} & . & . & . & . & \Heti{X_{1}^{2}X_{2}\,e_{3}} \\ 
. & . & . & a_{4} & . & b_{2} & . & b_{1} & . & . & \Heti{X_{1}X_{2}^{2}\,e_{3}} \\ 
. & . & . & a_{5} & . & b_{3} & b_{2} & . & b_{1} & . & \Heti{X_{1}X_{2}X_{3}\,e_{3}} \\ 
. & . & . & a_{7} & . & . & . & b_{2} & . & . & \Heti{X_{2}^{3}\,e_{3}} \\ 
. & . & . & a_{8} & . & . & . & b_{3} & b_{2} & . & \Heti{X_{2}^{2}X_{3}\,e_{3}} \\ 
. & . & . & a_{9} & . & . & . & . & b_{3} & b_{2} & \Heti{X_{2}X_{3}^{2}\,e_{3}} \\ 
}
$$
Vu le caractère hétéroclite de certaines colonnes, ce n'est pas
immédiat a priori que $\det B_{k,d}$ soit homogène en $P_i$. 
Dans cet exemple, $\det B_{2,d}$ ne dépend pas du
dernier polynôme $P_3 = c_{1}X_{1}^{2} + c_{2}X_{1}X_{2} +
c_{3}X_{1}X_{3} + c_{4}X_{2}^{2} + c_{5}X_{2}X_{3} + c_{6}X_{3}^{2}$:
cela se voit d'ailleurs à l'oeil nu. Pourquoi?

En guise de $B$-$W$-devinette: pour le même jeu $(P_1, P_2, P_3)$
et même degré $d=5$, on a
$$
W_{2,5} \ = \ 
\EastBordermatrix{
a_{1} & . & . & . & . & . & . & . & . & . & \Heti{X_{1}^{4}X_{2}} \\ 
a_{2} & a_{1} & . & . & . & . & . & . & . & . & \Heti{X_{1}^{3}X_{2}^{2}} \\ 
a_{3} & . & a_{1} & . & . & . & . & . & . & . & \Heti{X_{1}^{3}X_{2}X_{3}} \\ 
. & . & . & a_{1} & b_{1} & . & . & . & . & . & \Heti{X_{1}^{3}X_{3}^{2}} \\ 
a_{6} & . & a_{3} & a_{2} & b_{2} & b_{1} & . & . & . & . & \Heti{X_{1}^{2}X_{2}X_{3}^{2}} \\ 
a_{9} & a_{6} & a_{5} & a_{4} & . & b_{2} & . & b_{1} & . & . & \Heti{X_{1}X_{2}^{2}X_{3}^{2}} \\ 
a_{10} & . & a_{6} & a_{5} & . & b_{3} & b_{2} & . & b_{1} & . & \Heti{X_{1}X_{2}X_{3}^{3}} \\ 
. & a_{9} & a_{8} & a_{7} & . & . & . & b_{2} & . & . & \Heti{X_{2}^{3}X_{3}^{2}} \\ 
. & a_{10} & a_{9} & a_{8} & . & . & . & b_{3} & b_{2} & . & \Heti{X_{2}^{2}X_{3}^{3}} \\ 
. & . & a_{10} & a_{9} & . & . & . & . & b_{3} & b_{2} & \Heti{X_{2}X_{3}^{4}} \\ 
}
$$
Les deux matrices (carrées) $B_{2,5}$ et $W_{2,5}$ ont même taille et même
déterminant, en fait même polynôme caractéristique. Pourquoi?
D'après la section~\ref{SectionDeltakdMacaulay}, pour $n=3$ et $d \ge \delta+1$:
$$
\det B_{2,d} = \Delta_{2,d}\,\Delta_{3,d} = \det W_{2,d}\, \det W_{3,d}
$$
Ici pour le format $D = (3,1,2)$ et le degré $d = 5$, on a $\delta + 1 \leqslant d < d_1 + d_2 + d_3$
donc $\Jex_{3,d} = 0$ donc~\dots{}
Il s'agit ici d'un cas particulier de l'expression binomiale de
$\det B_{k,d}$ en fonction des $\big(\det W_{h,d}\big)_{h \geqslant k}$.
On verra dans le chapitre~\ref{ChapBW}
qu'en organisant correctement les bases monomiales, on obtient des structures triangulaires remarquables,
cf par exemple l'illustration de la page~\pageref{ExempleD312R1}.

\begin {rmq} [Le statut de l'anneau $\bfA$ des coefficients en ce qui concerne le poids en $P_i$]

Nous avions annoncé que l'étude du poids nécessitait un terrain
générique pour permettre d'étudier le degré en les coefficients de
chaque $P_i$ des divers déterminants.  Nous avons décrit une structure
$\bbN^n$-graduée de $\BW^\sbullet(\bfA[\uX]^n)$ comme $\bfA$-module en
tenant compte de la nature de~$\bfA$ comme anneau de polynômes sur un
petit anneau~$\bfk$.  Que constate-t-on dans les propositions
précédentes? Que le poids en~$P_i$ ne dépend que du format de degrés~$D$,
format permettant d'élaborer la primitive $\minDiv$, la décomposition de
Macaulay.  En fin de compte, tout se passe au niveau des monômes
extérieurs $X^\beta e_J$ et pas du tout du côté des coefficients
des~$P_i$, et donc l'anneau $\bfA$ joue un rôle mineur.
\end {rmq}

\subsubsection*{Décomposition de Macaulay tordue par $\sigma\in\fS_n$ : poids en $P_i$ des $\det B^\sigma_{k,d}(\uP)$} 

Le pendant de la proposition~\ref{PoidsDetBkdSum} concernant le poids en $P_i$ des
$\det B_{k,d}$ est le suivant:

\begin{prop}
Pour $\sigma \in \fS_n$, le déterminant $\det B^\sigma_{k,d}$ est homogène en $P_i$ et
$$
\poids_{P_i}(\det B^\sigma_{k,d}) 
\ = \ 
\#\big\{X^\beta e_J \in \Mmac^\sigma_{k-1,d} \ \big | \ \sminDiv(X^\beta) = i \big\}
\ =\ 
\#\big\{ X^\alpha e_I \in \Smac^\sigma_{k,d} \ \big | \ \smin I = i \big\}
$$
En particulier, $\det B^\sigma_{k,d}$ ne dépend pas des $P_i$ pour les $i$
figurant parmi les $k-1$ derniers éléments de $\big(\sigma(1), \dots, \sigma(n)\big)$.
\end{prop}

\index{de@décomposition de Macaulay de $\rmK_\sbullet(\uP)$!tordue par $\sigma\in\fS_n$}%

\begin {exemple}

Soit $\uP$ le système suivant de format $D = (2,2,2,1)$:
$$
\setlength{\tabcolsep}{2pt}
\left\{
\begin{tabular}{rcp{14cm}}
$P_{1}$ & $=$ & $a_{1}X_{1}^{2} + a_{2}X_{1}X_{2} + a_{3}X_{1}X_{3} + a_{4}X_{1}X_{4} + a_{5}X_{2}^{2} + a_{6}X_{2}X_{3} +
a_{7}X_{2}X_{4} + a_{8}X_{3}^{2} + a_{9}X_{3}X_{4} + a_{10}X_{4}^{2}$\\ [0.1cm]
$P_{2}$ & $=$ & $b_{1}X_{1}^{2} + b_{2}X_{1}X_{2} + b_{3}X_{1}X_{3} + b_{4}X_{1}X_{4} + b_{5}X_{2}^{2} + b_{6}X_{2}X_{3} +
b_{7}X_{2}X_{4} + b_{8}X_{3}^{2} + b_{9}X_{3}X_{4} + b_{10}X_{4}^{2}$\\ [0.1cm]
$P_{3}$ & $=$ & $c_{1}X_{1}^{2} + c_{2}X_{1}X_{2} + c_{3}X_{1}X_{3} + c_{4}X_{1}X_{4} + c_{5}X_{2}^{2} + c_{6}X_{2}X_{3} +
c_{7}X_{2}X_{4} + c_{8}X_{3}^{2} + c_{9}X_{3}X_{4} + c_{10}X_{4}^{2}$\\ [0.1cm]
$P_{4}$ & $=$ & $d_{1}X_{1} + d_{2}X_{2} + d_{3}X_{3} + d_{4}X_{4}$\\ [0.1cm]
\end{tabular}
\right.
$$
Voici une sous-matrice carrée de $\partial_{2,d}(\uP)$, avec $d=4$, extraite sur $T \times S$,
$S$ pour source (colonnes), $T$ pour target (lignes) où l'on note $\overline c_i$ pour $-c_i$ et
$\overline d_i$ pour $-d_i$:
$$
\let\ov=\overline
\NorthEastBordermatrix{
\Veti{e_{12}} & \Veti{e_{13}} & \Veti{X_{1}\,e_{14}} & \Veti{X_{2}\,e_{14}} & \Veti{X_{3}\,e_{14}} & \Veti{X_{4}\,e_{14}} &
\Veti{e_{23}} & \Veti{X_{1}\,e_{24}} & \Veti{X_{2}\,e_{24}} & \Veti{X_{3}\,e_{24}} & \Veti{X_{4}\,e_{24}} & \Veti{X_{1}\,e_{34}}
& \Veti{X_{2}\,e_{34}} & \Veti{X_{3}\,e_{34}} & \Veti{X_{4}\,e_{34}} & \\
a_{1} & . & . & . & . & . & \ov c_{1} & \ov d_{1} & . & . & . & . & . & . & . & \Heti{X_{1}^{2}\,e_{2}} \\
a_{4} & . & . & . & . & . & \ov c_{4} & \ov d_{4} & . & . & \ov d_{1} & . & . & . & . & \Heti{X_{1}X_{4}\,e_{2}} \\
a_{7} & . & . & . & . & . & \ov c_{7} & . & \ov d_{4} & . & \ov d_{2} & . & . & . & . & \Heti{X_{2}X_{4}\,e_{2}} \\
a_{8} & . & . & . & . & . & \ov c_{8} & . & . & \ov d_{3} & . & . & . & . & . & \Heti{X_{3}^{2}\,e_{2}} \\
a_{9} & . & . & . & . & . & \ov c_{9} & . & . & \ov d_{4} & \ov d_{3} & . & . & . & . & \Heti{X_{3}X_{4}\,e_{2}} \\
a_{10} & . & . & . & . & . & \ov c_{10} & . & . & . & \ov d_{4} & . & . & . & . & \Heti{X_{4}^{2}\,e_{2}} \\
. & a_{1} & . & . & . & . & b_{1} & . & . & . & . & \ov d_{1} & . & . & . & \Heti{X_{1}^{2}\,e_{3}} \\
. & . & a_{1} & . & . & . & . & b_{1} & . & . & . & c_{1} & . & . & . & \Heti{X_{1}^{3}\,e_{4}} \\
. & . & a_{2} & a_{1} & . & . & . & b_{2} & b_{1} & . & . & c_{2} & c_{1} & . & . & \Heti{X_{1}^{2}X_{2}\,e_{4}} \\
. & . & a_{3} & . & a_{1} & . & . & b_{3} & . & b_{1} & . & c_{3} & . & c_{1} & . & \Heti{X_{1}^{2}X_{3}\,e_{4}} \\
. & . & a_{4} & . & . & a_{1} & . & b_{4} & . & . & b_{1} & c_{4} & . & . & c_{1} & \Heti{X_{1}^{2}X_{4}\,e_{4}} \\
. & . & a_{8} & . & a_{3} & . & . & b_{8} & . & b_{3} & . & c_{8} & . & c_{3} & . & \Heti{X_{1}X_{3}^{2}\,e_{4}} \\
. & . & . & a_{8} & a_{6} & . & . & . & b_{8} & b_{6} & . & . & c_{8} & c_{6} & . & \Heti{X_{2}X_{3}^{2}\,e_{4}} \\
. & . & . & . & a_{8} & . & . & . & . & b_{8} & . & . & . & c_{8} & . & \Heti{X_{3}^{3}\,e_{4}} \\
. & . & . & . & a_{9} & a_{8} & . & . & . & b_{9} & b_{8} & . & . & c_{9} & c_{8} & \Heti{X_{3}^{2}X_{4}\,e_{4}} \\
}
$$
Nous allons donner plusieurs manières de déterminer le poids de ce déterminant. La première, sans
renseignement supplémentaire, est générale et consiste à utiliser le lemme~\ref{PoidsDeterminant}.
Puisque:
$$
\poids(S) = 6\varepsilon_1 + 6\varepsilon_2 + 6\varepsilon_3  + 12\varepsilon_4,
\qquad\qquad
\poids(T) = 6\varepsilon_2 + \varepsilon_3 + 8\varepsilon_4,
$$
le déterminant est de poids
$$
\poids(S) - \poids(T) = 6\varepsilon_1 + 5\varepsilon_3 + 4\varepsilon_4
$$
En particulier, il ne dépend pas de $P_2$. Si on utilise sur les
coefficients des $P_i$ l'ordre lexicographique vérifiant $a_1 > a_2 >
\cdots > b_1 > b_2 > \cdots > d_4$, le terme dominant du déterminant est
$a_1^6 c_8^5 d_4^4$.

\medskip
Mais en fait, renseignement supplémentaire, la sous-matrice en question est
$\beta^\sigma_{2,4}(\uP) : \Smac^\sigma_{2,4} \to \Mmac^\sigma_{1,4}$
avec $\sigma = (2,3,4)$. On peut donc commencer par affirmer, le dernier indice de
$$
\big(\sigma(1), \sigma(2), \sigma(3), \sigma(4)\big) = (1,3,4,2)
$$
étant 2, que le déterminant ne dépend pas de $P_2$ (comme prévu). La
proposition précédente donne le poids de $\det B^\sigma_{2,4}(\uP)$,
qui est aussi celui de $\det \beta^\sigma_{2,4}(\uP)$. Elle permet de
déterminer le poids en $P_i$ : soit par utilisation de la base monomiale
de $\Smac^\sigma_{2,4}$ (``bordure horizontale'') en comptant le
nombre de $X^\alpha\,e_I \in \Smac^\sigma_{2,4}$ tels que $\smin(I) =
i$ ; soit par utilisation de la base monomiale de $\Mmac^\sigma_{1,4}$
(``bordure verticale'') en comptant le nombre de $X^\beta\,e_J \in
\Mmac^\sigma_{1,4}$ tels que $\sminDiv(X^\beta) = i$.  On laisse le
soin au lecteur de réaliser les vérifications.
\end {exemple}

\subsection {Poids en $P_i$ des $\Delta_{k,d}(\protect\uP)$
             pour $k=1,2$ et quelques séries dimensionnelles}

On s'intéresse dans cette section au poids en $P_i$ de $\Delta_{k,d} =
\Delta_{k,d}(\uP)$ c.a.d. à l'étude d'égalités du type:
$$
\Delta_{k,d}(a_1P_1, \dots, a_nP_n) = a_1^{e_1} \cdots a_n^{e_n}\,
\Delta_{k,d}(P_1, \dots, P_n)
$$
pour des scalaires quelconques $a_1, \dots, a_n$, les $e_i$ étant des
entiers dépendant du format $D$ et de $k,d$. Il revient au même de
supposer $\uP$ générique et d'étudier le degré d'homogénéité en $P_i$
du polynôme $\Delta_{k,d}(\uP)$ de l'anneau $\bfA=\bfk[\indetsPi]$ des
coefficients du système $\uP$.

Comme on a l'égalité
$$
\Delta_{k,d} \times \big(\det B_{k+1,d}\, \det B_{k+3,d} \cdots\big) =
\det B_{k,d}\, \det B_{k+2,d} \cdots
$$
et que chaque $\det B_{k',d}$ est homogène en $P_i$ (proposition~\ref{PoidsDetBkdSum}),
il est clair que $\Delta_{k,d}$ est homogène en~$P_i$ et que son poids en $P_i$ est donné par
$$
\begin {array} {rcl}
\poids_{P_i}(\Delta_{k,d}) &=&
\poids_{P_i}(\det B_{k,d})  - \poids_{P_i}(\det B_{k+1,d}) +  \poids_{P_i}(\det B_{k+2,d}) - \cdots
\\ [3mm]
&=&\displaystyle \sum_{k'=k}^n (-1)^{k'-k}\,\poids_{P_i}(\det B_{k',d})
\\
\end {array}
$$
L'objectif de cette section est de déterminer ce poids pour $k = 1,2$.
Pour y parvenir, au lieu de fixer~$d$, nous allons au contraire, $k$ et $i$
étant fixés, considérer la famille $\big(\poids_{P_i}(\Delta_{k,d})\big)_d$ avec $d$ variant dans $\bbN$ pour former certaines séries de $\bbZ[[t]]$ et les
identifier en tant que fractions rationnelles explicites. 

\bigskip

On se propose de commencer par déterminer les séries des modules $E$ où $E = \Jex_{1\setminus 2}^{(i)}$, $\Jex_1^{(i)}$, $\Jex_2^{(i)}$. 
Par définition, nous noterons $\Serie(E)$ la série  $\sum_{d \ge 0} \dim E_d \, t^d$.

Nous poursuivrons avec la série du module $\Mmac_{k-1}^{(i)}$ 
(rappelons que pour $k = 1$, on a $\Mmac_{0} = \Jex_1$), qui sera définie par :
$$
\sum_{d \ge 0} \poids_{P_i}(\det B_{k,d})\,t^d
$$

Enfin, nous nous occuperons des séries 
$$
\sum_{d \ge 0} \poids_{P_i}(\Delta_{1,d})\,t^d, 
\qquad \text{ et } \qquad 
\sum_{d \ge 0} \poids_{P_i}(\Delta_{2,d})\,t^d 
$$

\subsubsection*{\fbox{Les séries de $\Jex_{1\setminus 2}^{(i)}$, $\Jex_1^{(i)}$, $\Jex_2^{(i)}$}}

Comme annoncé dans l'introduction, nous notons 
$$
\Serie(\Jex^{(i)}_{1\setminus2}) = \sum_{d \ge 0} \dim \Jex^{(i)}_{1\setminus2,d}\, t^d
$$
et l'on rappelle que $\Jex_{1\setminus2,d}^{(i)}$ est le sous-module monomial de
$\bfA[\uX]_d$ de base les monômes ayant $i$ comme seul indice de
divisibilité.

\begin {prop} [Série de $\Jex^{(i)}_{1\setminus2}$] \leavevmode
\label {SerieJ1minusJ2i}
  
\begin {enumerate}[\rm i)]
\item    
La série de $\Jex^{(i)}_{1\setminus2}$ est donnée par la fraction rationnelle suivante:
$$
\Serie\big(\Jex^{(i)}_{1\setminus2}\big)
\ = \ 
\dfrac {N_i(t)} {(1-t)^n}
\qquad \text{où} \qquad
N_i(t) = t^{d_i} \times \prod_{j\ne i} (1 - t^{d_j})
$$

\item
Notons $U_i(t) = t^{d_i}\times \prod_{j \ne i} (1-t^{d_j})/(1-t)$. C'est un polynôme
unitaire de degré $\delta+1$, à coefficients entiers $\ge 0$, certifiant que $t=1$
est un pôle d'ordre 1 de la fraction rationnelle ci-dessus via l'écriture:
$$ 
\Serie\big(\Jex^{(i)}_{1\setminus2}\big)
\ = \ 
\dfrac{U_i(t)}{1-t}, \qquad  U_i(1) = \widehat d_i \ne 0
$$
\end{enumerate}
\end {prop}

\begin {proof}
Pour $\alpha = (\alpha_1, \dots, \alpha_n)$, 
l'appartenance $X^\alpha \in \Jex_{1\setminus 2,d}^{(i)}$ est équivalente aux contraintes:
$$
\alpha_1 < d_1, \dots, \alpha_{i-1} < d_{i-1}, \qquad \alpha_i \ge d_i,  \qquad 
\alpha_{i+1} < d_{i+1}, \dots, \alpha_{n} < d_{n}, 
\qquad  |\alpha| = d
\leqno(\star)
$$
Nous affirmons que le nombre de $X^\alpha$, c'est-à-dire le nombre de $\alpha$, obéissant aux contraintes $(\star)$
ci-dessus est le coefficient en~$t^d$ de la série suivante:
$$
\overbrace{\dfrac{1-t^{d_1}}{1-t} \cdots \dfrac{1-t^{d_{i-1}}}{1-t}}^{\pi_1} \times
\overbrace{\dfrac{t^{d_i}}{1-t}}^{\pi_2} \times
\overbrace{\dfrac{1-t^{d_{i+1}}}{1-t} \cdots \dfrac{1-t^{d_{n}}}{1-t}}^{\pi_3}
$$
En effet, développons $\pi_1, \pi_2, \pi_3$ en séries formelles
($\pi_1$ et $\pi_3$ sont des polynômes):
$$
\left\{
\begin {array}{rcl}
\pi_1 &=& (1 + t + \cdots + t^{d_1-1}) \cdots (1 + t + \cdots + t^{d_{i-1}-1})
\\
\pi_2 &=& t^{d_i}  + t^{d_i+1}  + t^{d_i+2}  + \cdots
\\
\pi_3 &=& (1 + t + \cdots + t^{d_{i+1}-1}) \cdots (1 + t + \cdots + t^{d_{n}-1})
\\
\end {array}
\right.
$$
Pour identifier le coefficient en $t^d$ dans le développement du produit $\pi_1\pi_2\pi_3$,
on doit considérer une puissance de $t$ dans chaque facteur de $\pi_1$, puis une puissance de $t$ dans $\pi_2$, puis une puissance de $t$ dans chaque facteur de $\pi_3$ avec la condition que 
la somme des exposants de ces puissances fasse $d$ ;
autrement dit, on doit considérer tous les $\alpha = (\alpha_1, \dots, \alpha_n)$ avec  $|\alpha|= d$ et 
$$
\begin {array}{*{8}c}
t^{\alpha_1} &\cdots   &t^{\alpha_{i-1}} & t^{\alpha_i} & t^{\alpha_{i+1}} & \cdots & t^{\alpha_n} 
\\[2mm]
\text{pris dans le} &\cdots &\text{pris dans le} &\text{pris dans $\pi_2$} &
\text{pris dans le} &\cdots& \text{pris dans le}
\\
\text{premier} &\cdots &\text{dernier} &  &
\text{premier} &\cdots& \text{dernier}
\\
\text{facteur de $\pi_1$} &\cdots &\text{facteur de $\pi_1$} & &
\text{facteur de $\pi_3$} &\cdots& \text{facteur de $\pi_3$}
\end {array}
$$
Ces contraintes sur $\alpha$ sont exactement celles de~$(\star)$.

\medskip
Bilan: le nombre de $X^\alpha \in \Jex_{1,d}^{(i)}$ est donc le
coefficient en $t^d$ de la série
$$
\pi_1\pi_2\pi_3 
\ \overset{\rm def.}{=} \ \dfrac{N_i(t)}{(1-t)^n}
\qquad \text{avec} \qquad
N_i(t) =  t^{d_i} \times \prod_{j\neq i} (1-t^{d_j})
$$
Ceci démontre le point i). Et le point ii) n'offre pas de difficulté.
\end {proof}

On va maintenant déterminer la série de $\Jex_1^{(i)}$ 
et en déduire celle de $\Jex_{2}^{(i)}$. 
Ces deux modules sont implicitement pilotés
par $\minDiv$ (cela ne se voit pas dans la notation) ; $\Jex_h^{(i)}$ est le
sous-module monomial de $\Jex_{h}$ de base les monômes $X^\alpha
\in \Jex_{h}$ tels que $\minDiv(X^\alpha) = i$.
L'utilisation de $\dim\Jex^{(\minDiv = i)}_{h}$ aurait été plus adéquate
mais un tantinet trop lourde. 
Signalons que ce qui vient d'être dit ne concerne pas (directement)
$\Jex_{1\setminus 2}^{(i)}$ qui, bien qu'il puisse être défini à l'aide de $\minDiv$,
possède une définition intrinsèque, sans aucun recours à un mécanisme de sélection,
cf. le rappel avant l'énoncé de la proposition précédente.

\begin {prop} [Séries de $\Jex_1^{(i)}$ et $\Jex_2^{(i)}$]\leavevmode 
\label{J1iJ2iSeries}
Elles sont données par les fractions rationnelles:
$$
\Serie(\Jex^{(i)}_1) = \dfrac {\Niun(t)}{(1-t)^n}
\qquad \text{avec} \qquad
\Niun(t) = t^{d_i} \times \prod_{j<i} (1-t^{d_j})
$$
et
$$
\Serie(\Jex^{(i)}_2) = \dfrac {\Nideux(t)}{(1-t)^n} 
\qquad \text{avec} \qquad
\Nideux(t) = t^{d_i} \times \prod_{j<i} (1-t^{d_j}) \times
\Big(1 - \prod_{j>i} (1-t^{d_j}) \Big)
$$
\end {prop}

\begin{proof}
\leavevmode

$\rhd$ 
Pour $\alpha = (\alpha_1, \dots, \alpha_n)$, 
l'appartenance $X^\alpha \in \Jex_{1,d}^{(i)}$ est équivalente aux contraintes:
$$
\alpha_1 < d_1, \dots, \alpha_{i-1} < d_{i-1}, \qquad \alpha_i \ge d_i,  \qquad  |\alpha| = d
\leqno(\star)
$$
Nous affirmons que le nombre de $X^\alpha$, c'est-à-dire le nombre de $\alpha$, obéissant aux contraintes $(\star)$
ci-dessus est le coefficient en~$t^d$ de la série suivante:
$$
\overbrace{\dfrac{1-t^{d_1}}{1-t} \cdots \dfrac{1-t^{d_{i-1}}}{1-t}}^{\pi_1} \times
\overbrace{\dfrac{t^{d_i}}{1-t}}^{\pi_2} \times
\overbrace{\dfrac{1}{(1-t)^{n-i}}}^{\pi_3}
$$
En effet, développons $\pi_1, \pi_2, \pi_3$ en séries formelles
($\pi_1$ est un polynôme):
$$
\left\{
\begin {array}{rcl}
\pi_1 &=& (1 + t + \cdots + t^{d_1-1}) \cdots (1 + t + \cdots + t^{d_{i-1}-1})
\\
\pi_2 &=& t^{d_i}  + t^{d_i+1}  + t^{d_i+2}  + \cdots
\\
\pi_3 &=& (1 + \cdots + t^{\alpha_{i+1}} + \cdots) \cdots (1 + \cdots + t^{\alpha_{n}} + \cdots)
\\
\end {array}
\right.
$$
Pour identifier le coefficient en $t^d$ dans le développement du produit $\pi_1\pi_2\pi_3$,
on doit considérer une puissance de $t$ dans chaque facteur de $\pi_1$, puis une puissance de $t$ dans $\pi_2$, puis une puissance de $t$ dans chaque facteur de $\pi_3$ avec la condition que 
la somme des exposants de ces puissances fasse $d$ ;
autrement dit, on doit considérer tous les $\alpha = (\alpha_1, \dots, \alpha_n)$ avec  $|\alpha|= d$ et 
$$
\begin {array}{*{8}c}
t^{\alpha_1} &\cdots   &t^{\alpha_{i-1}} & t^{\alpha_i} & t^{\alpha_{i+1}} & \cdots & t^{\alpha_n} 
\\[2mm]
\text{pris dans le} &\cdots &\text{pris dans le} &\text{pris dans $\pi_2$} &
\text{pris dans le} &\cdots& \text{pris dans le}
\\
\text{premier} &\cdots &\text{dernier} &  &
\text{premier} &\cdots& \text{dernier}
\\
\text{facteur de $\pi_1$} &\cdots &\text{facteur de $\pi_1$} & &
\text{facteur de $\pi_3$} &\cdots& \text{facteur de $\pi_3$}
\end {array}
$$
Ces contraintes sur $\alpha$ sont exactement celles de~$(\star)$.

\medskip
Bilan: le nombre de $X^\alpha \in \Jex_{1,d}^{(i)}$ est donc le
coefficient en $t^d$ de la série
$$
\pi_1\pi_2\pi_3 
\ \overset{\rm def.}{=} \ \dfrac{\Niun(t)}{(1-t)^n}
\qquad \text{avec} \qquad
\Niun(t) =  t^{d_i} \times \prod_{j<i} (1-t^{d_j})
$$

\bigskip

$\rhd$ Passons à la preuve du second point.
Comme $\Jex_1^{(i)} = \Jex_2^{(i)} \oplus \Jex_{1\setminus 2}^{(i)}$, on a 
$$
\Serie(\Jex^{(i)}_2) = \Serie(\Jex^{(i)}_1) - \Serie(\Jex^{(i)}_{1\setminus2})
$$
Avec le point précédent et la proposition~\ref {SerieJ1minusJ2i}, on en déduit :
$$
\Serie(\Jex^{(i)}_2)
= \dfrac {\Niun - N_i}{(1-t)^n}
\qquad \text{avec} \qquad
\Niun = t^{d_i} \times \prod_{j<i} (1-t^{d_j}), \quad
N_i = t^{d_i} \times \prod_{j \ne i} (1 - t^{d_j})
$$
Dans le polynôme $N_i$, le produit peut s'écrire :
$$
\prod_{j \ne i} (1 - t^{d_j}) =  \prod_{j<i} (1 - t^{d_j}) \times \prod_{j>i} (1 - t^{d_j})
$$
D'où :
$$
\Niun - N_i = t^{d_i} \times \prod_{j<i} (1-t^{d_j}) \times \Big(1 - \prod_{j>i} (1-t^{d_j}) \Big)
$$
La formule annoncée est donc démontrée !
\end {proof}

\subsubsection{Compléments}

Avec ce qui précède, on peut facilement obtenir les séries de $\Jex_1$, $\Jex_2$ et $\Jex_{1\setminus2}$.
Il suffit de sommer pour $i \in \llbracket 1,n \rrbracket$ 
les séries précédentes. 
Par exemple, 
$$
\Serie(\Jex_1) = \dfrac {\sum_{i=1}^n \Niun(t)}{(1-t)^n}  
\qquad \text{ et } \qquad 
\Serie(\Jex_2) = \dfrac {\sum_{i=1}^{n-1} \Nideux(t)}{(1-t)^n}  
$$
En ce qui concerne la série de $\Jex_2$, on pourra remarquer la sommation jusqu'à $n-1$ au
lieu de $n$, ceci étant dû au fait que $\Nndeux = 0$ (dans l'expression de $\Nndeux$, le troisième facteur
est nul), en accord avec $\Jex_2^{(n)} = 0$.

\subsubsection{\fbox{La série du poids en $P_i$ de $\det B_{k,\sbullet}$, ou encore du module $\Mmac_{k-1}^{(i)}$}}

Suite au résultat de la proposition~\ref{PoidsDetBkdSum} affirmant:
$$
\poids_{P_i}(\det B_{k,d}) 
\ = \ 
\#\big\{X^\beta e_J \in \Mmac_{k-1,d} \ \big | \ \minDiv(X^\beta) = i \big\}
\ = \ 
\#\big\{ X^\alpha e_I \in \Smac_{k,d} \ \big | \ \min(I) = i \big\}
$$
et au fait que $B_k$ est un endomorphisme de $\Mmac_{k-1}$, 
nous définissons un sous-$\bfA$-module monomial de~$\Mmac_{k-1}$,
noté $\Mmac_{k-1}^{(i)}$:
$$
\Mmac_{k-1}^{(i)} = \text{sous-module de base les 
$X^\beta e_J \in \Mmac_{k-1}$ tels que $\minDiv(X^\beta) = i$}
$$
Sa série est par définition:
$$
\Serie\big(\Mmac_{k-1}^{(i)}\big) \overset{\rm def.}{=}
\sum_{d \ge 0} \dim \Mmac_{k-1,d}^{(i)}\, t^d  =
\sum_{d \ge 0} \poids_{P_i}(\det B_{k,d})\, t^d 
$$
Nous l'identifions ci-après en tant que fraction rationnelle.

\label {NOTA14-Mk-1i}%

Puisque $\Jex_1 = \Mmac_0$, nous allons retrouver des éléments des preuves précédentes dans ce qui suit.

\begin {prop} [La série du poids en $P_i$ de $\det B_{k,\sbullet}$ ou encore du module $\Mmac_{k-1}^{(i)}$] \leavevmode
\label{BkdPoidsPiSerie}

Pour $L \subset \{1..n\}$, posons $d_L = \sum_{\ell \in L} d_\ell$ et pour $k \ge 1$, considérons 
le polynôme $\Aik(t) \in \bbZ[t]$ défini de deux manières différentes:
$$
\Aik(t) 
\ = \ 
t^{d_i}\!\!\! \sum_{J\subset\{i+1..n\}\atop\#J=k-1} \kern -8pt t^{d_J} \ =\ 
\sum_{\min(I)=i\atop\#I=k} \kern -6pt t^{d_I}
$$
Alors:
$$
\Serie\big(\Mmac_{k-1}^{(i)}\big) = \dfrac{N_{i,k}(t)}{(1-t)^n}
\qquad \text{avec} \qquad
N_{i,k}(t) = \Aik(t) \times \prod_{j<i} (1-t^{d_j})
$$
\end {prop}

\begin {proof}
\leavevmode

\medskip
$\rhd$
Soit $I \subset \{1..n\}$ tel que $\min(I) = i$ et $\#I = k$; alors $J := I \setminus \{i\}$
est contenu dans $\{i+1..n\}$ et est de cardinal $k-1$. Réciproquement, soit $J \subset
\{i+1..n\}$ tel que $\#J = k-1$; alors $I := \{i\} \sqcup J$ est de cardinal~$k$
et vérifie $\min(I) = i$. Et l'égalité
$$
t^{d_I} = t^{d_i}\, t^{d_J}
$$
prouve celle des deux expressions définissant le polynôme $\Aik$.

\medskip
$\rhd$
Pour $\beta = (\beta_1, \dots, \beta_n)$ et $J \subset \{1..n\}$, 
l'appartenance $X^\beta e_J \in \Mmac_{k-1,d}^{(i)}$ est équivalente aux contraintes:
$$
\beta_1 < d_1, \dots, \beta_{i-1} < d_{i-1}, \qquad \beta_i \ge d_i, \qquad
J\subset\{i+1..n\},\quad \#J = k-1, \qquad  |\beta| + d_J = d
\leqno(\star)
$$
\emph{Fixons} $J$ vérifiant $J\subset\{i+1..n\}$ et $\#J = k-1$.  Nous
affirmons que le nombre de $X^\beta e_J$, c'est-à-dire le nombre de
$\beta$, obéissant aux contraintes $(\star)$ ci-dessus est le
coefficient en~$t^d$ de la série suivante:
$$
\overbrace{\dfrac{1-t^{d_1}}{1-t} \cdots \dfrac{1-t^{d_{i-1}}}{1-t}}^{\pi_1} \times
\overbrace{\dfrac{t^{d_i}}{1-t}}^{\pi_2} \times
\overbrace{\dfrac{1}{(1-t)^{n-i}}}^{\pi_3} \times\ t^{d_J}
$$
En effet, développons $\pi_1, \pi_2, \pi_3$ en séries formelles
($\pi_1$ est un polynôme):
$$
\left\{
\begin {array}{rcl}
\pi_1 &=& (1 + t + \cdots + t^{d_1-1}) \cdots (1 + t + \cdots + t^{d_{i-1}-1})
\\
\pi_2 &=& t^{d_i}  + t^{d_i+1}  + t^{d_i+2}  + \cdots
\\
\pi_3 &=& (1 + \cdots + t^{\beta_{i+1}} + \cdots) \cdots (1 + \cdots + t^{\beta_{n}} + \cdots)
\\
\end {array}
\right.
$$
Pour identifier le coefficient en $t^d$ dans le développement du
produit $\pi_1\pi_2\pi_3\,t^{d_J}$, on doit considérer une puissance
de $t$ dans chaque facteur de $\pi_1$, puis une puissance de $t$ dans
$\pi_2$, puis une puissance de $t$ dans chaque facteur de $\pi_3$ avec
la condition que la somme des exposants de ces puissances et de $d_J$
fasse $d$ ; autrement dit, on doit considérer tous les $\beta =
(\beta_1, \dots, \beta_n)$ avec $|\beta| + d_J = d$ et
$$
\begin {array}{*{8}c}
t^{\beta_1} &\cdots   &t^{\beta_{i-1}} & t^{\beta_i} & t^{\beta_{i+1}} & \cdots & t^{\beta_n} 
\\[2mm]
\text{pris dans le} &\cdots &\text{pris dans le} &\text{pris dans $\pi_2$} &
\text{pris dans le} &\cdots& \text{pris dans le}
\\
\text{premier} &\cdots &\text{dernier} &  &
\text{premier} &\cdots& \text{dernier}
\\
\text{facteur de $\pi_1$} &\cdots &\text{facteur de $\pi_1$} & &
\text{facteur de $\pi_3$} &\cdots& \text{facteur de $\pi_3$}
\end {array}
$$
Ces contraintes sur $\beta$ (toujours à $J$ fixé) sont exactement celles de~$(\star)$.

Avec des notations qui se devinent (sinon, le lecteur  attendra le prochain chapitre !), on a donc, à~$J$ fixé :
$$
\Serie\big(\Mmac_{k-1}^{(i)}[J] \big) 
\ = \ 
\pi_1 \pi_2 \pi_3 t^{d_J}
\ = \ 
\dfrac{\text{num}_J}{(1-t)^n}
\qquad \text{avec} \qquad
\text{num}_J \ = \ 
t^{d_i} \, t^{d_J} \times \prod_{j<i} (1-t^{d_j})
$$
Pour obtenir le résultat, il suffit maintenant de sommer sur les 
parties $J$ de $\{i+1,\dots, n\}$  de cardinal~\mbox{$k-1$}.
Par définition du polynôme $\Aik$, on a donc :
$$
\Serie\big(\Mmac_{k-1}^{(i)}\big)
\ = \ 
\dfrac{\Aik(t)\times \prod\limits_{j<i} (1-t^{d_j})}{(1-t)^n}
\ \overset{\rm def.}{=} \ \dfrac{N_{i,k}(t)}{(1-t)^n}
$$
\end {proof}

\subsubsection*{\fbox{Séries du poids en $P_i$ de $\Delta_{1,\sbullet}$ et $\Delta_{2,\sbullet}$}}

Nous sommes en mesure d'énoncer le résultat principal de cette section.
Pour cela, nous aurons besoin du lemme élémentaire (voire trivial) ci-dessous que nous énonçons sans preuve.

\begin {lem} \leavevmode

Pour un ensemble fini $F$ et deux familles $(a_j)_{j \in F}$,   $(b_j)_{j \in F}$,
à valeurs dans un anneau commutatif:
$$
\prod_{j \in F}(a_j + b_j) = \sum_{J \subset F}  \Big(\prod_{j'\notin J} a_{j'} \prod_{j\in J} b_j\Big)
$$  
\end  {lem}

\begin {theo} [La série du poids en $P_i$ de $\Delta_{1,\sbullet}$] \leavevmode 
\label{poidsDelta1}

\begin {enumerate} [\rm i)]
\item
Voici la série du poids en $P_i$ de $\Delta_{1,\sbullet}$ sous forme de fraction rationnelle:
$$
\sum_{d \ge 0} \poids_{P_i}(\Delta_{1,d})\, t^d = \dfrac {N_i(t)} {(1-t)^n}
\qquad \text{où} \qquad
N_i(t) = t^{d_i} \times \prod_{j\ne i} (1 - t^{d_j})
$$
Ainsi, 
$$
\sum_{d \ge 0} \poids_{P_i}(\Delta_{1,d})\, t^d 
\ =\ 
\Serie\big(\Jex^{(i)}_{1\setminus2}\big)
$$

\item  
Le scalaire $\Delta_{1,d} = \Delta_{1,d}(\uP)$ est
homogène en $P_i$ de poids $\dim\Jex^{(i)}_{1\setminus2,d}$.  En
particulier:
$$
\poids_{P_i}(\Delta_{1,d}) = \begin {cases}
\widehat d_i &\text{si $d \ge \delta+1$} \\[0.2cm]
\widehat d_i -1 &\text{si $d=\delta$} \\
\end{cases}
$$
\end {enumerate}
\end {theo}

\begin {proof} 
\leavevmode

\noindent 
i) 
En utilisant l'égalité suivante (expliquée au début de cette section) :
$$
\poids_{P_i}(\Delta_{1,d}) = \sum_{k=1}^n (-1)^{k-1} \poids_{P_i}(\det B_{k,d})
$$
il vient par définition:
$$
\sum_{d \ge 0} \poids_{P_i}(\Delta_{1,d})\, t^d = \sum_{k=1}^n (-1)^{k-1} \Serie\big(\Mmac_{k-1}^{(i)}\big)
$$
D'après la proposition~\ref{BkdPoidsPiSerie}, la somme dans le membre droit ci-dessus est égale à:
$$
\dfrac{N'_i(t)}{(1-t)^n}
\qquad \text{avec} \qquad
N'_i(t) = \left(\sum_{k=1}^n (-1)^{k-1} \Aik(t)\right) \times \prod_{j<i} (1-t^{d_j})
$$
Pour obtenir l'expression convoitée de la série, il suffit de prouver que
$$
\sum_{k=1}^n (-1)^{k-1} \Aik(t) \ = \ t^{d_i} \, \prod_{j > i} (1 - t^{d_j})
\leqno (\star)
$$
car il en résultera que $N'_i = N_i$.

L'application du lemme élémentaire précédent à $F = \{i+1..n\}$, $a_j=1$ et $b_j = -t^{d_j}$ fournit
$$
\prod_{j > i} (1 - t^{d_j}) \ = \sum_{J \subset \{i+1..n\}} \kern -7pt (-1)^{\#J} t^{d_J} 
$$
Dans la somme de droite, en regroupant les parties $J$ de même cardinal $k-1$:
$$
\prod_{j > i} (1 - t^{d_j}) \ =\ 
\sum_{k=1}^n (-1)^{k-1} \kern -8pt \sum_{J\subset\{i+1..n\}\atop\#J=k-1} t^{d_J}
$$
D'où:
$$
t^{d_i}\, \prod_{j > i} (1 - t^{d_j}) \ = \ 
\sum_{k=1}^n (-1)^{k-1} t^{d_i} \kern -8pt \sum_{J\subset\{i+1..n\}\atop\#J=k-1} t^{d_J} \ \overset{\rm def.}{=}\ 
\sum_{k=1}^n (-1)^{k-1} \Aik(t)
$$
ce qui prouve $(\star)$ et termine l'obtention de la formule annoncée.

\noindent
Une fois établie l'égalité des deux séries, on obtient le point~ii), ce qui termine la preuve du théorème.
\end {proof}

\begin {coro} \leavevmode
\label{poidsDelta2}
  
Le scalaire $\Delta_{2,d} = \Delta_{2,d}(\uP)$ est homogène en $P_i$ de poids $\dim\Jex^{(i)}_{2,d}$.
Ou encore:
$$
\poids_{P_i}(\Delta_{2,d}) = \poids_{P_i}(\det W_{2,d}) 
$$
En particulier, $\Delta_{2,d}(\uP)$ ne dépend pas de $P_n$.

\medskip
Plus généralement, pour $\sigma \in \fS_n$:
$$  
\poids_{P_i}(\Delta^\sigma_{2,d}) = \dim\Jex^{(\sminDiv = i)}_{2,d} =
\poids_{P_i}(\det W^\sigma_{2,d}) 
$$
En particulier, $\Delta^\sigma_{2,d}(\uP)$ ne dépend pas de $P_{\sigma(n)}$.
\end {coro}

\begin {proof}
On a par définition:
$$
\det B_{1,d} = \Delta_{1,d}\,\Delta_{2,d}
$$  
Or $\det B_{1,d} = \det W_{1,d}$ est homogène en $P_i$ de poids $\dim \Mmac_{0,d}^{(i)} = \dim \Jex^{(i)}_{1,d}$ (d'après~\ref{BkdPoidsPiSerie} ou d'après~\ref{PoidsDetW}), 
et $\Delta_{1,d}$ est homogène de poids $\dim \Jex^{(i)}_{1\setminus2,d}$ 
d'après~\ref{poidsDelta1}.
Comme $\Jex_1 = \Jex_{1\setminus2} \oplus \Jex_2$,  on a l'égalité dimensionnelle :
$$
\dim \Jex_{1,d}^{(i)} 
\ = \ 
\dim \Jex^{(i)}_{1\setminus2,d}
+ 
\dim \Jex^{(i)}_{2,d}
$$
Ainsi, $\Delta_{2,d}$ est homogène en $P_i$ de poids $\dim \Jex^{(i)}_{2,d}$.
Comme $\dim \Jex^{(n)}_{2,d} = 0$, on en déduit que $\Delta_{2,d}$
ne dépend pas de $P_n$.

\medskip
Pour $\sigma \in \fS_n$ quelconque, on utilise l'égalité analogue 
$$
\det B^\sigma_{1,d} = \Delta^\sigma_{1,d}\,\Delta^\sigma_{2,d} 
$$
et le fait que $\Delta_{1,d} = \Delta^\sigma_{1,d}$
(cf.~\ref{EgaliteDelta1dDelta1dsigma}).
\end {proof}

La vraie vérité sera révélée ultérieurement. Elle consiste en l'égalité évoquée à plusieurs reprises:
$$
\forall d \ge 0, \quad \forall \sigma \in \fS_n, \qquad 
\Delta^\sigma_{2,d}(\uP) = \det W^\sigma_{2,d}(\uP) $$
et demande beaucoup plus de travail.

\medskip

Le contrôle du poids en $P_i$ de $\Delta_{1,d}(\uP)$ permet la détermination de la forme $\Det_d$
du premier cas d'école étudié en~\ref{soussectionPremierCasEcole}.

\begin {prop} [La forme $\Det_d$ du format $D = (1,\dots, 1,e)$]
\label{DetdCasEcoleI}  
\leavevmode
  
Soit $\uP = (L_1, \cdots, L_{n-1}, P_n)$ un système de format $D = (1,\dots,1,e)$
(les $L_i$ sont donc des formes linéaires) et
$\uxi = (\xi_1, \cdots, \xi_n)$ défini en \ref{soussectionPremierCasEcole} par
le déterminant dans la base $(X_1, \dots, X_n)$ de $\bfA[\uX]_1$
$$
\xi_i = \det(L_1, \cdots, L_{n-1}, X_i)
$$
Pour $d \le \delta$:

\begin {enumerate}[\rm i)]
\item  
On a $\Smac_{0,d} = \bfA X_n^d$ et la forme $\Det_d$ s'identifie
à une forme linéaire $\bfA[\uX]_d \to \bfA$.

\item
La forme $\Det_d$ est la restriction à $\bfA[\uX]_d$ de la forme d'évaluation en $\xi$:
$$
\Det_d = (\evalxi)_{|\bfA[\uX]_d}
$$
A fortiori $\Delta_{1,d} = \xi_n^d$. 
\end {enumerate}
\end {prop}

\begin {proof} \leavevmode

Le premier point ne présente pas de difficulté. Remarquons, puisque $d \le \delta=e-1$,
que:
$$
\Im\Syl_d(\uP) \overset{\rm def.}{=} \langle \uP\rangle_d = \langle \uL\rangle_d
$$
Comme $\calS_{0,d}=\{X_n^d\}$, remarquons également que, dans
l'isomorphisme $\AXdSod \simeq \bfA[\uX]_d$, l'injection canonique
$\iota: \calS_{0,d} \to \bfA[\uX]_d$ s'identifie à $X_n^d$.

\medskip

En ce qui concerne le second point, on peut supposer $\uL$ générique.
Donc, d'après le point iii) de~\ref{evalxiProperties}, on a
$\Gr(\xi_1, \xi_n) \ge 2$. On en déduit $\Gr(\xi_1^d, \xi_n^d) \ge 2$,
a fortiori
$$
\Gr\big((\evalxi)_{|\bfA[\uX]_d}\big) \ge 2
$$
D'après la proposition~\ref{FormeLineaireGr2},
la forme $(\evalxi)_{|\bfA[\uX]_d}$ est donc une $\bfA$-base du sous-module
des formes linéaires sur $\bfA[\uX]_d$ nulles sur $\langle\uP\rangle_d =
\langle\uL\rangle_d$.

\smallskip

Or la forme $\Det_d$ passe au quotient sur $\bfB_d$ i.e. est nulle sur $\langle \uP\rangle_d$.
En conséquence, il existe $\lambda \in \bfA = \bfk[\indetsPi]$, homogène en poids, tel que:
$$
\Det_d = \lambda \times (\evalxi)_{|\bfA[\uX]_d}
$$
Evaluons ces formes en $\iota \simeq X_n^d$:
$$
\Delta_{1,d} = \lambda \times \xi_n^d
$$
Examinons le poids en $L_i$ et $P_n$. A gauche, celui de
$\Delta_{1,d}$ est $\dim \Jex_{1\setminus2,d}^{(i)}$; or, la base
monomiale de $\Jex_{1\setminus2,d}^{(i)}$ est $(X_i^k X_n^{d-k})_{1\le
  k\le d}$.  Donc $\poids_{L_i}(\Delta_{1,d}) = \dim
\Jex_{1\setminus2,d}^{(i)} = d$ tandis que $\poids_{P_n}(\Delta_{1,d})
= \dim \Jex_{1\setminus2,d}^{(n)} = 0$.  A droite, le poids de $\xi_n$
en $L_i$ est 1 et celui en $P_n$ est~0.  Bilan: $\lambda$ est homogène
de poids~$0$ en $L_i$ et $P_n$ i.e. $\lambda \in \bfk$.  La
spécialisation en le jeu étalon $(X_1, \dots, X_{n-1}, X_n^e)$ de
$\Delta_{1,d}$ est 1 et celle de $\xi_n$ est 1. Bilan: $\lambda=1$, ce
qui termine la preuve.
\end {proof}

\medskip

A l'aide de cette proposition, il est facile de construire des suites \emph {régulières}
$\uP = (L_1, \dots, L_{n-1}, P_n)$ de format $(1, \dots, 1,e)$ telles que
$\Delta_{1,d}(\uP) = 0$ pour $1 \le d \le \delta$: il suffit de forcer $\xi_n = 0$.
Prenons par exemple $n=3$, $L_1 = a_1X_1 + a_3X_3$ et $L_2 = b_1X_1 + b_3X_3$. Les
$\xi_i$ sont donnés par
$$
\begin {vmatrix}
a_1 &b_1 &T_1 \\
.   &.   &T_2 \\  
a_3 &b_3 &T_3 \\
\end {vmatrix} =
\xi_1 T_1 + \xi_2 T_2 + \xi_3 T_3,
\qquad\qquad
\xi_1 = \xi_3 = 0, \quad
\xi_2 = -\begin {vmatrix} a_1 &b_1\\ a_3 &b_3 \\ \end{vmatrix}
$$
C'est cette construction que nous avons utilisée dans la section \ref{SectionDeltakdMacaulay}
pour produire des suites régulières~$\uP$ telles que
$\Delta_{1,d}(\uP) = 0$ pour tout $1 \le d \le \delta$.

\subsection{L'inégalité $\Gr\big(\det B_{k,d}^\sigma(\protect\uP), \sigma\in\fS_n\big) \geqslant k$ et
son impact sur les $\Delta_{k,d}(\protect\uP)$}
\label{sectionProfondeurBk}

Pour un système $\uP$ couvrant le jeu étalon généralisé $\pXD$
(cf.~\ref{JeuCouvrantEtalon} pour la définition), un premier objectif
est d'établir l'inégalité de profondeur:
$$
\Gr \big(
\det B_{k,d}^\sigma(\uP), \ 
\sigma \in \fS_n
\big) 
\ \geqslant \ k
$$
Ce résultat est totalement indépendant des autres sections et ne nécessite que la définition
de $B_{k,d}^\sigma$. Nous en déduirons la même inégalité sur les
$\Delta^\sigma_{k,d}(\uP)$ et retrouverons l'égalité $\Delta_{1,d}(\uP) = \Delta_{1,d+1}(\uP)$
pour $d \ge \delta+1$.

\medskip

Un commentaire avant d'en fournir la preuve.
D'après~\ref{JeuCouvrantEtalonRegularite}, la suite $\uP$ est
régulière donc (cf. \ref{ExactitudeKd}) le complexe de $\bfA[\uX]$-modules
$\rmK_{\sbullet}(\uP)$ est exact.  Il en est alors de même de chaque
composante homogène \idest{} de chaque complexe de $\bfA$-modules
libres $\rmK_{\sbullet, d}(\uP)$.
Un des théorèmes fondamentaux de la théorie des résolutions libres
finies, le théorème~\ref{WhatMakesAComplexExact}
(What makes a complex exact?), fournit, pour l'idéal
déterminantiel~$\fD_{k,d}(\uP)$ d'ordre $r_{k,d}$ de la
différentielle~$\partial_{k,d}(\uP)$, l'inégalité
$\Gr\big(\fD_{k,d}(\uP) \big) \geqslant k$.  Or d'après la définition
de $B^\sigma_{k,d}$, chaque $\det B_{k,d}^\sigma(\uP)$ est un mineur
d'ordre $r_{k,d}$ de $\partial_{k,d}(\uP)$ de sorte que:
$$
\langle \det B_{k,d}^\sigma(\uP), \sigma \in \fS_n \rangle \subset \fD_{k,d}(\uP)
$$
En conséquence, l'inégalité visée est plus précise que
celle donnée par le théorème~\ref{WhatMakesAComplexExact}. Il faut cependant
noter que, d'une part, ce dernier théorème est bien plus général et que
d'autre part, l'hypothèse requise ici sur $\uP$ est plus forte
(couvrir le jeu étalon généralisé versus régulière).

Précisons que la preuve du théorème qui vient est indépendante de la
théorie des résolutions libres finies. Elle pourrait donc permettre
la mise en place de la structure multiplicative du complexe
$\rmK_{\sbullet,d}(\uP)$, du moins lorsque $\uP$ couvre le jeu étalon
généralisé.

\medskip

Nous reprenons des notations analogues à celles de la
section~\ref{sectionProfondeurWh} en notant $\kderniers$
l'ensemble des $k-1$ derniers éléments de $\{1..n\}$, de cardinal
$n-(k-1)$, et $\Sigma_k$ n'importe quel sous-ensemble de $\fS_n$
vérifiant:
$$
\hbox {toute partie de $\{1..n\}$ de cardinal $\#\kderniers$ est de la forme $\sigma(\kderniers)$ pour
  un certain $\sigma \in \Sigma_k$}
$$
Nous y avions montré que:
$$
\Gr \big(
\det W_{k,d}^\sigma(\uP), \ 
\sigma \in \Sigma_k
\big) 
\ \geqslant \ k
$$
L'objectif est de montrer le résultat analogue suivant:

\begin{theo}\label{ProfondeurBk}
Soit $\uP$ un système couvrant le jeu étalon généralisé (comme par exemple le système générique).
Alors, pour chaque $d$:
$$
\Gr \big(
\det B_{k,d}^\sigma(\uP), \ 
\sigma \in \Sigma_k
\big) 
\ \geqslant \ k
$$
\end{theo}

\begin{proof} \leavevmode

L'anneau $\bfA$ des coefficients de $\uP$ est un anneau de polynômes
de la forme $\bfA=\bfR[p_1, \ldots, p_n]$ où les~$p_i$ sont les
coefficients du jeu étalon généralisé $\pXD = (p_1X_1^{d_1}, \dots,
p_nX_n^{d_n})$.

\index{idéal!de Veronese}%

\medskip
$\bullet$
Nous traitons d'abord le cas du jeu étalon généralisé $\pXD$ lui-même.
On rappelle que, d'une part, on a noté~$\scrV_r$ l'idéal monomial de
Veronese de $\bfR[p_1, \ldots, p_n]$ engendré par les monômes sans
facteur carré de degré~$r$ et que d'autre part, on a montré en
\ref{ProfondeurVeronese} l'inégalité:
$$
\Gr \big(\scrV_{n-(k-1)}\big) \, \geqslant\, k
$$
Par définition, ce $n-(k-1)$, c'est $\#\kderniers$ et l'ensemble $\{
\sigma(\kderniers),\ \sigma \in \Sigma_k \}$ est celui des parties
$\calP_{\#\kderniers}\big(\{1..n\}\big)$. On peut donc écrire:
$$
\scrV_{\#\kderniers} 
\ = \ 
\langle \, p_{\sigma(\kderniers)},\ \sigma \in \Sigma_k \, \rangle \qquad \text{où} \qquad
p_{\sigma(\kderniers)} \ =\ \prod_{i\in \sigma(\kderniers)} p_i
$$
Soit $\sigma \in \fS_n$.  Puisque la matrice de $B_{k,d}^\sigma(\pXD)$
est diagonale avec des $p_j$ sur la diagonale:
$$
\det B_{k,d}^\sigma(\pXD) 
\ = \ 
\prod_{i \in \sigma(\kderniers)} p_i^{\nu_{i,d}}  
$$
où $\nu_{i,d}$ est la dimension du $\bfA$-module 
de base les $X^\beta e_J \in \Mmac_{k-1,d}$ tels que $\sminDiv(X^\beta) = i$.

En prenant $e$ plus grand que tous les $\nu_{i,d}$, $i$ variant dans $\sigma(\kderniers)$
et $\sigma$ dans $\Sigma_k$, on a donc
$$
\langle \,  p_{\sigma(\kderniers)}^e,\ \sigma \in \Sigma_k \, \rangle
\subset
\langle  \det B_{k,d}^\sigma(\pXD), \  \sigma \in \Sigma_k \rangle 
$$
D'après~\ref{ProfondeurVeronese}, on a $\Gr(p_{\sigma(\kderniers)},\ \sigma \in \Sigma_k) \ge k$,
donc (corollaire~\ref{GrCompatibleInclusionRacine}),
$\Gr(p_{\sigma(\kderniers)}^e,\ \sigma \in \Sigma_k) \ge k$. 
On termine en utilisant $\fb \subset \fa \Rightarrow \Gr(\fa) \ge \Gr(\fb)$.

\medskip
$\bullet$
Traitons maintenant le cas où $\uP$ couvre $\pXD$.  Examinons la
composante homogène dominante de $\det B_{k,d}^\sigma(\uP)$, polynôme
de $\bfR[p_1, \dots, p_n]$.  Un coefficient de $B_{k,d}^\sigma(\uP)$
est de degré $1$ s'il est sur la diagonale, et de degré $0$ sinon (un
coefficient de la diagonale est un certain $p_i$ et les coefficients
hors de la diagonale sont dans $\bfR$).  Par conséquent, dans le
développement de $\det B^\sigma_{k,d}(\uP)$, le terme indexé par une
permutation $\tau$ est, dans $\bfR[p_1, \dots, p_n]$, un monôme de
degré le nombre de points fixes de $\tau$.  La composante homogène
dominante de $\det B_{k,d}^\sigma(\uP)$ est donc réduite au terme
correspondant à $\tau = \Id$.  C'est exactement $\det
B_{k,d}^\sigma(\pXD)$.  On conclut en utilisant le contrôle de la profondeur par
les composantes homogènes dominantes, cf. le théorème~\ref{ControleProfondeur}.
\end{proof}

\begin {coro}
Soit $\uP$ un système couvrant le jeu étalon généralisé (comme par exemple le système générique).
Alors, pour tout $k \ge 1$ et tout $d \in \bbN$:
$$
\Gr \big(\det \Delta_{k,d}^\sigma(\uP),\ \sigma \in \Sigma_k\big) 
\ \geqslant \ k
$$
En particulier, $\Gr \big(\det \Delta_{2,d}^\sigma(\uP),\ \sigma \in \Sigma_2\big) \ \geqslant \ 2$.
\end {coro}

\begin {proof}
Par définition, en omettant $\uP$ des notations:
$$
\det B^\sigma_{k,d} = \Delta^\sigma_{k,d} \Delta^\sigma_{k+1,d}
\qquad \text{d'où} \qquad
\fa := \big\langle\det B_{k,d}^\sigma, \ \sigma \in \Sigma_k\big\rangle
\subset
\fb := \big\langle \Delta_{k,d}^\sigma, \ \sigma \in \Sigma_k\big\rangle
$$
On termine en utilisant $\Gr(\fa) \geqslant k$ (théorème précédent) et 
$\fa \subset \fb \Rightarrow \Gr(\fb) \ge \Gr(\fa)$.
\end {proof}

\subsubsection*{Une \og troisième\fg{} preuve de l'égalité $\Delta_{1,d}=\Delta_{1,d+1}$ pour $d\ge\delta+1$}

Pour $d \ge\delta+1$, nous montrons maintenant l'égalité $\Delta_{1,d}
= \Delta_{1,d+1}$ \emph {de manière autonome}.  Nous n'utilisons pas
le fait que $\Delta_{1,d}$ est l'invariant de MacRae de $\bfB_d$, ni,
a fortiori, le résultat~\ref{MacRaeEqualities} affirmant que $\bfB_d$,
$\bfB_{d+1}$ ont même invariant de MacRae.
Nous allons utiliser l'égalité $\Delta^\sigma_{1,d} = \Delta_{1,d}$,
valide pour tout $d \in \bbN$ et tout $\sigma \in \fS_n$ (cf.~\ref{EgaliteDelta1dDelta1dsigma}
qui pointe en fait sur le lemme de rigidité~\ref{RigidityLemma}).
Il s'agit donc d'une
justification interne à la machinerie ``décompositions de Macaulay''.
Nous avons mis un s à décomposition car nous allons utiliser les
décompositions de Macaulay tordues par $\sigma \in \fS_n$.
Nous pouvons considérer que l'égalité $\Delta_{1,d} = \Delta_{1,d+1}$
repose d'une part sur
$$
\det B^\sigma_{k,d} = \Delta^\sigma_{k,d} \Delta^\sigma_{k+1,d}
\quad \text{pour $k=1,2$}
\qquad\qquad
\Gr \big(\det \Delta_{2,d}^\sigma,\ \sigma \in \Sigma_2\big) \ \geqslant \ 2
\text{ (en générique)}
$$
et d'autre part sur le contrôle du poids de $\Delta_{1,d}$ en $P_i$.

\smallskip

Notre justification diffère de celle de Demazure dans son
article \textit{Une définition constructive du
  résultant}~\cite{Demazure1}. Ce dernier utilise une récurrence sur $n$
(il ne dispose pas des décompositions de Macaulay tordues par
une permutation).
Certes, cette récurrence n'est pas notre tasse de thé. Cependant,
c'est un \emph {pur calcul} tandis que notre approche \emph {ici} utilise
un \emph {raisonnement} (profondeur 2, contrôle du poids) et 
présente donc un caractère indirect.

Néanmoins, l'utilisation de l'égalité $\Delta_{2,d} = \det W_{2,d}$
(cf \ref{sousSectionMacaulayRecurrence} pour $d \ge \delta$ et le
théorème~\ref{DeltakdProdBinomialDetWhd} pour tout $d$) change la
vision du scalaire $\Delta_{1,d}$ que l'on peut exprimer alors sous la
forme $\Delta_{1,d} = \det W_{1,d} / \det W_{2,d}$. Auquel cas, la
section \ref{sectionEgaliteStabiliteMacRae} (Stabilité de l'invariant
de MacRae) fournit une autre justification de type ``pur calcul''
(sans récurrence sur $n$) de l'égalité $\Delta_{1,d} = \Delta_{1,d+1}$
pour $d \ge \delta+1$.

\begin {theo}
\label{StabiliteDelta1d}  
Pour $d \ge\delta+1$, on a $\Delta_{1,d} = \Delta_{1,d+1}$.
\end {theo}

\begin {proof}
Nous pouvons supposer $\uP$ générique.  

$\bullet$
Montrons d'abord que $\Delta_{1,d} \mid \Delta_{1,d+1}$. Nous disposons de la relation de divisibilité
prouvée en~\ref{DetW1dDiviseDetW1dplus1}:
$$
\det W^\sigma_{1,d}  \mid  \det W^\sigma_{1,d+1} 
$$
Ecrivons alors (la dernière égalité utilisant $\Delta_{1,d}=\Delta^\sigma_{1,d}$ vue en~\ref{EgaliteDelta1dDelta1dsigma}):
$$
\det W^\sigma_{1,d} = \det B^\sigma_{1,d} = \Delta^\sigma_{1,d} \Delta^\sigma_{2,d} =
\Delta_{1,d} \Delta^\sigma_{2,d}
$$
et l'égalité analogue pour $d+1$. La relation de divisibilité initiale s'écrit alors:
$$
\Delta_{1,d} \Delta^\sigma_{2,d}  \mid \Delta_{1,d+1} \Delta^\sigma_{2,d+1}
\qquad \text{a fortiori (sans s'inquiéter de faire relâche!)} \qquad
\Delta_{1,d}  \mid \Delta_{1,d+1} \Delta^\sigma_{2,d+1}
$$
Comme $\Gr\big(\Delta^\sigma_{2,d+1},\ \sigma\in \fS_n) \ge 2$, on en déduit
$\Delta_{1,d}  \mid \Delta_{1,d+1}$.

\medskip
$\bullet$
Ecrivons $\Delta_{1,d+1} = \lambda\,\Delta_{1,d}$ où $\lambda \in \bfA = \bfk[\indetsPi]$ est
homogène en chaque~$P_i$. Comme $\poids_{P_i}(\Delta_{1,d}) = \poids_{P_i}(\Delta_{1,d+1}) = \widehat d_i$
d'après le théorème~\ref{poidsDelta1}, $\lambda$ est en fait dans $\bfk$. En spécialisant~$\uP$
en le jeu étalon, on obtient $\lambda = 1$.
\end {proof}

\smallskip

Suite à l'égalité $\det W^\sigma_{1,d} = \Delta_{1,d} \Delta^\sigma_{2,d}$,
nous pouvons énoncer, lorsque $\uP$ couvre le jeu étalon généralisé,
que $\Delta_{1,d}$ est le pgcd fort des $\big(\det W^\sigma_{1,d}\big)_{\sigma \in \Sigma_2}$.
Cela permet une légère variante dans la justification 
de $\Delta_{1,d} \mid \Delta_{1,d+1}$: considérer les relations de divisibilité
$\det W^\sigma_{1,d} \mid \det W^\sigma_{1,d+1}$ et \og en prendre le pgcd fort\fg. 

\medskip

Le théorème précédent permet de retrouver le résultat de~\ref{MacRaeEqualities} en degré $\ge \delta+1$.

\begin {coro}
Lorsque $\uP$ est régulière, les $\bfA$-modules $(\bfB_d)_{d \ge \delta+1} $ de MacRae
de rang 0 ont même invariant de MacRae.
\end {coro}  

\cleardoublepage

\section{Relations binomiales entre les familles $\big(\det B_{k,d}(\protect\uP)\big)_k$ et
  $\big(\det W_{h,d}(\protect\uP)\big)_h$}
\label{ChapBW}

Soit $D = (d_1, \dots, d_n)$ un format de degrés. 
Un des objectifs est de décomposer le module de Macaulay~$\Mmac_{k-1}$ pour $k \ge 1$
en sous-modules monomiaux isomorphes à l'idéal monomial $\Jex_h$ pour $h \ge k$. Précisément,
on va obtenir l'isomorphisme de $\bfA$-modules :
$$
\Mmac_{k-1} \ \simeq \ 
\bigoplus\limits_{h = k}^n  \ \Jex_h^{\oplus e_{k,h}}
\qquad 
\text{où \ $e_{k,h} = \binom{h-2}{h-k}$}\qquad\qquad
\text{\scriptsize égal à $\scriptstyle\binom{h-2}{k-2}$ sauf pour $k=1$}
$$
On souhaite également et surtout, pour tout système $\uP$ de format $D$, 
que l'induit-projeté de $B_k(\uP)$ sur chaque exemplaire de $\Jex_h$
soit quasi-conjugué à l'endomorphisme $W_h(\uP)$ de $\Jex_h$. 

\index{quasi-conjugués (endomorphismes)}%

\bigskip

Le cas $k=1$, qui se résume à $\Mmac_0 = \Jex_1$, est de peu
d'intérêt.  L'exemple qui vient ensuite pour $k=2$ nous semble simple
et pertinent: c'est celui du sous-module de Macaulay $\Mmac_1$. Il
s'agit en un premier temps d'en décrire une décomposition
$\bfA$-monomiale (les modules qui interviennent sont donc en
particulier des $\bfA$-modules $\bbN^n$-gradués)
$$
\Mmac_1 \simeq \Jex_2 \oplus \Jex_3 \oplus \cdots \oplus \Jex_n
\leqno (\star)
$$
Les objets qui interviennent pour l'instant ne dépendent que du format de 
degrés $D = (d_1, \ldots, d_n)$. 
En un deuxième temps, nous ferons intervenir un système homogène quelconque $\uP =
(P_1, \ldots, P_n)$ de format~$D$; en désignant très provisoirement
$B_{2}^{(h)}$ l'induit-projeté de $B_{2}$ sur
l'exemplaire de~$\Jex_h$ dans $\Mmac_1$ pour
$2 \leqslant h \leqslant n$, nous montrerons alors que
les endomorphismes $B_{2}^{(h)}$ et $W_{h}$ sont quasi-conjugués et en particulier 
$\det B_{2,d}^{(h)} = \det W_{h,d}$. Nous serons alors en mesure de prouver:
$$
\det B_{2,d} =  \det W_{2,d}\,\det W_{3,d}\, \cdots\, \det W_{n,d} 
$$

\medskip

Commençons par décrire~$(\star)$ en considérant les injections monomiales
$\Jex_h \to \rmK_1$ pour $h \geqslant 2$ sur le modèle :
$$
\begin {array}{l}
\Jex_2 \to \rmK_1, \quad X^\alpha \mapsto \dfrac {X^\alpha}{X_j^{d_j}}\, e_j
\hbox { où $j$ est le dernier élément de $\DivSeq(X^\alpha)$}
\\ [0.5cm]
\Jex_3 \to \rmK_1, \quad X^\alpha \mapsto \dfrac {X^\alpha}{X_j^{d_j}}\, e_j
\hbox { où $j$ est l'avant-dernier élément de $\DivSeq(X^\alpha)$}
\\
\end {array}
$$
Le lecteur doit comprendre que l'injection $\Jex_h \to \rmK_1$ est
donnée par $X^\alpha \mapsto \frac {X^\alpha}{X_j^{d_j}}\, e_j$ où $j$
est, dans $\DivSeq(X^\alpha)$, le \hbox{$(h\moins 1)$-ième} élément
\textit{en partant du plus grand}.
Prenons $h=n$ par exemple. Soit $X^\alpha \in \Jex_n$. Dans
$\DivSeq(X^\alpha) = \{1..n\}$, le $(n-1)$-ième élément en partant du
plus grand est $j = 2$ et l'injection $\Jex_n \to \rmK_1$ est
$X^\alpha \mapsto \frac {X^\alpha}{X_2^{d_2}}\, e_2$.

\smallskip

La première propriété, facile à constater, est que l'image de chaque
injection $\Jex_h \to \rmK_1$ est contenue dans~$\Mmac_1$. En effet,
en notant $i$ le $h$-ième élément (en partant du plus grand) de
$\DivSeq(X^\alpha)$, 
on a $i < j$ et le monôme $X^\beta = X^\alpha/X_j^{d_j}$ est divisible par $X_i^{d_i}$. 
D'où $X^\beta e_j \in \Mmac_1$.

\smallskip

La deuxième propriété est qu'il n'y a pas de \og collisions dans les images \fg{}. 
En effet, si $X^\beta e_j  \in \Mmac_1$ s'écrit 
$X^\alpha/X_j^{d_j} e_j$, 
il est clair que l'on retrouve $X^\alpha$ par l'égalité $X^\alpha = X_j^{d_j} X^\beta$ 
et on retrouve également~$h$ grâce à la position de $j$ dans $\DivSeq(X^\alpha)$. 
De manière précise, si $X^\beta e_j$ est un monôme extérieur de $\Mmac_1$, 
il est dans l'image d'une et d'une seule injection $\Jex_h \to \Mmac_1$, 
à savoir celle correspondant à $h$
obtenu en ajoutant~$1$ à la position de $j$ dans $\DivSeq(X_j^{d_j} X^\beta)$ ; 
\og position \fg{} est toujours à prendre dans le sens \textit{en partant du plus grand}.
On verra que les notations générales conduisent à écrire la décomposition $(\star)$ sous la forme :
$$
\Mmac_1 \ =\  
\Mmac_1^{\{1\}} \oplus \Mmac_1^{\{2\}} \oplus \cdots \oplus \Mmac_1^{\{n-1\}}
$$
$R= \{1\}$ étant un indicateur du dernier élément, $R = \{2\}$ de l'avant-dernier \dots{}
Se termine ainsi un aperçu de la description de $(\star)$. 

\bigskip

De manière analogue, on dispose d'une décomposition monomiale :
$$
\Mmac_2 \ \simeq \ 
\Jex_3 \oplus \Jex_4^{\oplus 2} \oplus 
\Jex_5^{\oplus 3} \oplus \cdots \oplus 
\Jex_n^{\oplus (n-2)}
$$
Prenons $n =4$ et évoquons l'injection $\Jex_3 \to \Mmac_2$ et les deux injections de $\Jex_4 \to \Mmac_2$.
Elles sont toutes de la même forme 
$$
X^\alpha \ \longmapsto \ \dfrac{X^\alpha}{X_i^{d_i}X_j^{d_j}} \ e_i \wedge e_j 
\qquad 
\hbox{avec $i <j$}
$$
où il convient de préciser les deux éléments $i$ et $j$ de $\DivSeq(X^\alpha)$, 
ou encore leur position $R$ dans la suite renversée $\DivG(X^\alpha)$.
$$
\begin{array}{l|l|l}
\hbox{L'injection $\Jex_3 \to \Mmac_2$} & 
\hbox{$i =$ avant-dernier et $j =$ dernier} & R = \{2,1\} 
\\ [0.2cm] 
\hbox{La première injection $\Jex_4 \to \Mmac_2$} & 
\hbox{$i =$ avant-avant-dernier et $j =$ dernier} & R = \{3,1\}
\\ [0.2cm] 
\hbox{La deuxième injection $\Jex_4 \to \Mmac_2$} & 
\hbox{$i =$ avant-avant-dernier et $j =$ avant-dernier} & R = \{3,2\}
\end{array}
$$
Plus généralement, pour décrire notre décomposition monomiale de $\Mmac_{k'}$
pour $k' \geqslant 1$, 
il est impératif d'être plus formel et d'introduire quelques notions élémentaires
relatives aux ensembles finis totalement ordonnés, à savoir les primitives $\Ind$ et $\Elt$, 
cf.~\ref{soussectionEltInd}.

\medskip

Il est important de signaler que, dans la suite de ce chapitre, la primitive $\minDiv$
intervient de manière presque invisible: les objets principaux en scène sont
$B_{k,d}(\uP)$ et $W_{h,d}(\uP)$, pilotés de manière implicite par $\minDiv$.

\medskip

\centerline{\fbox{\parbox{0.95\linewidth}
{
En ce qui concerne les notations, nous avons choisi l'indice $h$ pour le monde
excédentaire: $\Jex_h$, $\Jex_{h,d}$, $W_h$, $W_{h,d}$ et ses dérivés. Et
nous avons réservé l'indice $k$
pour les endomorphismes liés à la décomposition de Macaulay tels que
$B_k$, $B_{k,d}$ et leurs induits. Il faut compter avec le petit désagrément
résidant dans le fait que $B_k$ est un endomorphisme du sous-module
de Macaulay~$\Mmac_{k-1}$. 
Pour contrer ce pataquès $k \leftrightarrow k-1$, nous utiliserons
parfois le nom $\Mmac_{k'}$ quand ce dernier module est étudié pour
lui-même sans référence à $B_\sbullet$.
}
}}

\newpage


\subsubsection{Tableau résumé $B_k \leftrightsquigarrow W_h$, relations binomiales} 

Voici une présentation des divers ingrédients que nous allons mettre en place
dans l'objectif de montrer la relation binomiale encadrée (en bas). Ce n'est pas l'objet
dans ce résumé de définir les notations utilisées: cela sera
réalisé au fur et à mesure. On peut cependant préciser
que $J,R$ sont des parties de $\{1,\dots,n\}$ de même cardinal $k-1$.
Via cette page à laquelle le lecteur pourra se reporter en avançant dans sa lecture,
nous espérons ainsi offrir une vue d'ensemble.
$$
\psi :\ 
\begin{array}[t]{rcl}
\Mmac_{k-1} & \longrightarrow & \Jex_k \\[0.4em]
X^\beta e_J & \longmapsto & X^\beta X^{D(J)}
\qquad
\text{où \ $X^{D(J)} \,= \, \prod_{j \in J} X_j^{d_j}$}
\end{array}
$$
$$
\begin{array}{l|r}
k \ =\  1+\#J \ =\  1+\#R & h \ =\ 1+\max(R) \\  [0.6cm]
\Mmac_{k-1} =  \bigoplus\limits_R \Mmac^R  = 
\bigoplus\limits_R \bigoplus\limits_J \Mmac^R[J] & 
\bigoplus\limits_{h \geqslant k}  \Jex_h^{\oplus e_{k,h}} = 
\bigoplus\limits_R \Jex^R = \bigoplus\limits_R \bigoplus\limits_J \Jex^R[J] \\ [0.6cm]
X^\beta e_J \in \Mmac^R \ \Leftrightarrow \ \Ind\big( J \,; \DivSeq(\XbJ) \big) = R 
&\Jex^R\overset{\rm def.}{=} \Jex_{1+\max(R)},\quad W^R \overset{\rm def.}{=}W_{1+\max(R)}
\\ [0.6cm]
X^\beta e_{I} \in \Mmac^R[J] \ \Leftrightarrow \ I = J \text{ et } X^\beta e_I \in \Mmac^R
& X^\alpha \in \Jex^R[J] \ \Leftrightarrow \ \Elt\big( R \,; \DivSeq(X^\alpha) \big) = J \\ [0.6cm]
\hbox{$B_k(\uP)$ triangulaire pour $\Mmac_{k-1} = \bigoplus\limits_R \Mmac^R$ 
et $\Mmac_{k-1} = \bigoplus\limits_J \Mmac[J]$} & \\ [0.6cm]
\hbox{$B^R(\uP)$ triangulaire pour $\Mmac^R = \bigoplus\limits_J \Mmac^R[J]$} & 
\hbox{$W^R(\uP)$ triangulaire pour $\Jex^R = \bigoplus\limits_J \Jex^R[J]$}
\end{array}
$$

\medskip
\noindent
Au niveau $R$, entre $\Mmac^R$ et $\Jex^R = \Jex_h$,
la restriction $\psi^R$ du pont $\psi$ établit un isomorphisme:
$$
\psi^R : \begin{array}[t]{rcl}
\Mmac^R & \overset{\simeq}{\longrightarrow} & \Jex^R \\ [0.2cm]
X^\beta e_J & \longmapsto & X^\beta X^{D(J)}
\end{array}
\quad \text{d'inverse} \quad 
(\psi^R)^{-1} : \begin{array}[t]{rcl}
\Jex^R & \overset{\simeq}{\longrightarrow} & \Mmac^R \\ [0.2cm]
X^\alpha & \longmapsto & \dfrac{X^\alpha}{X^{D(J)}}e_J 
\quad \text{où $J = \Elt\big(R \,; \DivSeq(X^\alpha) \big)$}
\end{array}
$$
Tenons maintenant compte du système $\uP$. Au niveau $(R,J)$-induit-projeté, un rectangle commutatif:
$$
\def \Binduitproj{\begin{array}{c}
     \hbox {$B^R[J]$ induit-proj.} \\ \hbox {de $B_k(\uP)$} \end {array}}
\def \Winduitproj{\begin{array}{c}
     \hbox {$W^R[J]$ induit-proj.} \\ \hbox {de $W_h(\uP)$} \end {array}}
\xymatrix @R = 2cm @C = 5cm @M=0.4pc{
\Mmac^R[J] \ar[d]|-{\Binduitproj} \ar[r]^{\psi^R_J}_{\simeq}
& \Jex^R[J] \ar[d]|-{\Winduitproj}
\\
\Mmac^R[J] \ar[r]^{\psi^R_J}_{\simeq} & \Jex^R[J]
\\
}
$$
Dans la suite, on omet $\uP$ des notations. Du côté des déterminants, à $d$ fixé :
$$
\det B^R_{k,d}[J] = \det W^R_{h,d}[J]
$$
En réalisant le produit sur les $J$
$$
\det B^R_{k,d} = \det W^R_{h,d}
$$
En réalisant le produit sur les $J$ puis sur les $R$
$$
\det B_{k,d} \ =\  \prod\limits_R \prod\limits_J \det B_{k,d}^R[J] \ =\
\prod\limits_R \prod\limits_J \det W_{h,d}^R[J]  \ =\  
\prod\limits_R \det W_{h,d}^R \ =\  
\prod\limits_{h \geqslant k} \big(\det W_{h,d}\big)^{e_{k,h}}
$$
Cette dernière égalité mérite d'être encadrée:
$$
\boxed{    
\det B_{k,d} \ = \  
\prod_{h\geqslant k} \big(\det W_{h,d}\big)^{e_{k,h}}
\qquad \hbox{\ où \  }
e_{k,h} = \binom{h-2}{h-k}
}
$$

\index{relations binomiales}%


%
%
\label{NOTA15-ekh}%
%
%

\subsection{Le pont $\psi$ du complexe de Koszul $\rmK_\sbullet(\protect\uX^D)$ vers le degré homologique $0$}

Ce pont $\psi : \rmK_\sbullet(\uX^D) \to \rmK_0 = \bfA[\uX]$ est le $\bfA[\uX]$-morphisme gradué qui réalise
$$
\psi : X^\beta e_J \longmapsto \XbJ \qquad 
\text{où \ $X^{D(J)} \,= \, \prod_{j \in J} X_j^{d_j}$}
$$
Le lecteur comprendra l'utilisation du mot ``pont'' une fois pris connaissance du lemme~\ref{psiRiso}.

\begin{prop}
Le pont $\psi$ est un $\bfA[\uX]$-morphisme gradué qui envoie le
sous-module de Macaulay~$\Mmac_{k-1}$ sur l'idéal excédentaire~$\Jex_k$
$$
\psi :\ 
\begin{array}[t]{rcl}
\Mmac_{k-1} & \longrightarrow & \Jex_k \\[0.4em]
X^\beta e_J & \longmapsto & \XbJ
\end{array}
$$
\end{prop}

\label {NOTA15-Pontpsi}%
%
%

\begin{proof}
Montrons que $X^\alpha := \XbJ$ est bien un monôme de $\Jex_k$.
Par construction, $\DivSeq(X^\alpha)$ contient~$J$ de cardinal~\mbox{$k-1$}
ainsi que l'entier $i := \minDiv(X^\beta)$ n'appartenant pas à $J$ 
(on a $i < \min J$ car $X^\beta e_J$ est dans le module de Macaulay). 
Ainsi $\#\DivSeq(X^\alpha) \geqslant (k-1) + 1 = k$.

Réciproquement, soit $X^\alpha \in \Jex_k$. Notons $J$ l'ensemble de
ses $k-1$ derniers indices de divisibilité et posons $X^\beta = X^\alpha/X^{D(J)}$.
Alors $X^\beta e_J \in \Mmac_{k-1}$ et son image par $\psi$ est $X^\alpha$.
\end{proof}

Dans la suite, on va restreindre $\psi$ à un sous-$\bfA$-module
monomial $\calM$ de $\Mmac_{k'}$ qui n'est pas nécessairement un
$\bfA[\uX]$-module ; ainsi $\psi$ agit comme morphisme de
$\bfA$-modules.

\subsection{La $J$-décomposition et ses conséquences}

Le morphisme $\psi$ n'est pas injectif. 
Cependant, il le devient si, pour un $J \subset\{1..n\}$ fixé, on le restreint au sous-module de base
les monômes extérieurs \og ayant même $e_J$\fg{}, 
ce qui nous conduit à la définition suivante.
Dans cette section, $\calM$ désigne un sous-$\bfA$-module monomial de $\Mmac_{k'}$. 

\begin{defn} \label{DefMJ}
Pour $J \subset \{1..n\}$, notons $\calM[J]$ le sous-module monomial de $\calM$ 
de base les monômes extérieurs $X^\beta e_J$ appartenant à $\calM$.
On dispose ainsi d'une décomposition
$$
\calM = \bigoplus\limits_{\#J=k'} \calM[J] \qquad
\text{appelée {\rm \Jdecomposition}}
$$
En prenant $\calM = \Mmac_{k'}$ tout entier, on peut se permettre,
pour une partie $J \subset \{1..n\}$ de cardinal $k'$, 
de noter $\Mmac[J]$ au lieu de $\Mmac_{k'}[J]$ puisque
l'information du cardinal $k'$ de $J$ est ancrée dans $J$.
En bref, la construction $\Mmac[J]$ a le sens de $\Mmac_{\#J}[J]$.
\end{defn}

\label {NOTA15-MJ}%
\label {NOTA15-MmacJ}%

\medskip

Sur les parties de $\{1..n\}$ de même cardinal, on note $J
\preccurlyeq J'$ la relation d'ordre terme à terme lorsque l'on
ordonne ces parties de manière croissante (cf. la
section~\ref{sousSectionBkdJeuCirculaire} et la proposition
\ref{BkdQTriangSup} dont on reprend ci-dessous des éléments).
On rappelle également que dans cette section~\ref{sousSectionBkdJeuCirculaire},
nous avons défini l'endomorphisme~$B_\calM$ d'un sous-module monomial
$\calM$ de $\Mmac_{k'}$.

\label {NOTA15-BPcalM}%
%
%

\begin{prop} \label{JDecompositionEndo}
Soit $J \subset \{1..n\}$.

\begin{enumerate}[\rm i)]
\item
L'application $\psi_J$, restriction du pont $\psi$ à $\calM[J]$, est injective,
d'application réciproque:
$$
\psi_J^{-1} : \begin{array}[t]{rcl}
\psi(\calM[J]) & \longrightarrow & \calM[J] \\ [0.2cm]
X^\alpha & \longmapsto & \dfrac{X^\alpha}{X^{D(J)}}e_J
\qquad \text{où } X^{D(J)} \,= \, \prod_{j\in J} X_j^{d_j}
\end{array}
$$

\item
L'endomorphisme $B_\calM$ est triangulaire relativement à la \Jdecomposition :
$$
B_{\calM}\big( \calM[J] \big) 
\ \subset \ 
\bigoplus_{J' \preccurlyeq J} \calM[J']
$$ 

\item 
L'endomorphisme $B_{\calM[J]}$ est conjugué par $\psi_J$ à l'endomorphisme $W_{\psi(\calM[J])}$.
\end{enumerate}
\end{prop}

\begin{proof} \leavevmode

i) Il s'agit d'une simple vérification.

\medskip
ii) Soit $X^\beta e_J \in \calM[J]$ et $i := \minDiv(X^\beta) < \min J$.
D'après~\ref{ExpressionBkXbetaeJ}, 
$B_\calM(X^\beta e_J)$ est une combinaison $\bfA$-linéaire de monômes 
du type $X^{\beta'}e_{J'} \in \calM[J']$ avec
ou bien $J' = J$ ou bien $J' = i \vee J \setminus j$ pour $j \in J$.
Dans ce deuxième cas, on~a $J \succcurlyeq J'$ comme le montre le schéma ci-dessous 
(où $\ell$ est l'indice de $j$ dans $J$ i.e. $j = j_\ell$) :
$$
\begin{array} {*{18}c}
J : &  
j_1 & < & j_2 & < & \cdots & < &
   j_{\ell-1} & < & j_{\ell} & < & j_{\ell+1} & < & \cdots & < & j_{k-2} & < & j_{k-1}
\\ [0.5em]
J' : &
i & < & j_1 & < & \cdots & < & 
   j_{\ell-2} & < & j_{\ell-1} & < & j_{\ell+1} & < &\cdots & < & j_{k-2} & < & j_{k-1}
\\
\end{array}
$$

iii)
Soit $X^\beta e_J \in \calM[J]$. 
Posons $i = \minDiv(X^\beta)$. Comme $i < \min J$, on a $i = \minDiv(\XbJ)$.
Grâce à la formule donnée en~\ref{ExpressionBkXbetaeJ}, 
on a $B_{\calM[J]}(X^\beta e_J) = 
\pi_{\calM}\big(\frac{X^\beta}{X_i^{d_i}}\, P_i \, e_J\big)$.
On a donc le diagramme suivant dans lequel il reste à vérifier que la flèche $\psi_J$, 
en pointillé en bas, 
envoie l'élément de gauche sur l'élément de droite.
C'est bien le cas : si on applique $\psi_J^{-1}$ à l'élément de droite, 
on obtient l'élément de gauche.
$$
\xymatrix @M=0.6pc @R =1cm @C=1.5cm{
X^\beta e_J \ar[r]^-{\psi_J} \ar[d]_-{B_{\calM[J]}} &
\XbJ  \ar[d]^-{W_{\psi(\calM[J])}} \\
\pi_{\calM} \Big( 
\dfrac{X^\beta}{X_i^{d_i}}\, P_i \, e_J
\Big)
\ar@{-->}[r]^-{\psi_J} & 
\pi_{\psi(\calM[J])} \Big(\dfrac{\XbJ}{X_i^{d_i}} \,P_i\Big) \\
}
$$
\end{proof}

\begin{rmq}
Appliquons la définition~\ref{DefMJ} à $\calM = \Mmac_{k'}$ tout entier.
Le module $\Mmac_{k'}[J]$ est décrit en tant que $\bfA$-module.
Mais c'est en fait un $\bfA[\uX]$-module : il est engendré par la famille des monômes extérieurs
$(X_i^{d_i} e_J)_{i < \min J}$.
Comme $\psi$ est $\bfA[\uX]$-linéaire, l'image de $\Mmac_{k'}[J]$ par $\psi$
est également un $\bfA[\uX]$-module,  c'est-à-dire un idéal de $\bfA[\uX]$ ; 
c'est l'idéal monomial engendré par les $X_i^{d_i}X^{D(J)}$ avec $i < \min J$.

\medskip
Pour $J = \emptyset$, on a $\Mmac_0[\emptyset] = \Mmac_0 = \Jex_1$, la restriction
de $\psi$ à $\Mmac_0$ est l'identité en cohérence avec le fait que
les endomorphismes $B_1$ et $W_1$ coïncident.

\smallskip
Pour $J = \{2..n\}$, on a $\Mmac_{n-1}[J] = \Mmac_{n-1}$ et l'image de $\Mmac_{n-1}$ par $\psi$ 
est l'idéal $\Jex_n$.
Par conséquent, $B_n$ et $W_n$ sont conjugués par $\psi$, résultat déjà
développé en~\ref{BnEgalWn}.
\end{rmq}

\begin{rmq}
Pour un sous-module monomial $\calM$ \textit{quelconque}, 
les endomorphismes $B_\calM$ et $W_{\psi(\calM)}$ ne sont pas nécessairement conjugués par $\psi$, 
même lorsque la restriction de $\psi$ à $\calM$ est injective.
Il est donc indispensable que $\calM$ soit inclus dans un $\Mmac_{k'}[J]$.

\smallskip

Considérons $\calM = (X_1^{d_1} X_3^{d_3} e_2, \, X_2^{d_1+d_2} e_3)$
avec la contrainte $d_2 \geqslant d_1$. 
L'application $\psi : \calM \rightarrow \bfA[\uX]$ 
est injective car $\psi(\calM) = (X_1^{d_1} X_2^{d_2} X_3^{d_3}, \,
 X_2^{d_1+d_2} X_3^{d_3})$. 
Examinons maintenant la propriété de commutativité du diagramme 
en partant du monôme extérieur $X^\beta e_J = X_1^{d_1} X_3^{d_3} e_2$ 
(la démonstration est identique pour l'autre monôme extérieur $X_2^{d_1+d_2} e_3$).
On a $B_\calM(X^\beta e_J) = \pi_{\calM} \big( 
X_3^{d_3} P_1 e_2 - X_3^{d_3} P_2e_1
\big)
=
\pi_{\calM} \big( 
X_3^{d_3} P_1 e_2
\big)$,
la dernière égalité étant due au fait que $\calM$ ne contient pas de monôme extérieur en $e_1$.
Le diagramme s'écrit donc :
$$
\xymatrix @M=0.6pc @R =1cm @C=1.5cm{
 X_1^{d_1} X_3^{d_3} e_2 \ar[r]^-{\psi} \ar[d]_-{B_\calM} &
 X_1^{d_1} X_2^{d_2} X_3^{d_3}   \ar[d]^-{W_{\psi(\calM)}} \\
\pi_{\calM} \big( X_3^{d_3} P_1 e_2\big)
\ar@{-->}[r]^-{\psi} & 
\pi_{\psi(\calM)}\big(X_2^{d_2}X_3^{d_3} P_1\big) \\
}
$$
Il serait tentant de dire que la flèche en pointillé du bas envoie l'élément de gauche 
sur l'élément de droite. Mais il n'en est rien : à gauche, il n'y a qu'un seul monôme 
et à droite, il y en a deux, comme le montre la première colonne des matrices 
ci-dessous
$$
B_\calM =
\EastBordermatrix{
p_1 & . & \Heti{X_1^{d_1}X_3^{d_3}e_2} \\ 
\noalign{\vskip4pt} 
. & p_2 & \Heti{X_2^{d_1+d_2}e_3} \\ 
}
\qquad \qquad \qquad
W_{\psi(\calM)} =
\EastBordermatrix{
p_1 & q_2 & \Heti{X_1^{d_1} X_2^{d_2} X_3^{d_3}} \\ 
\noalign{\vskip4pt} 
q_1 & p_2 & \Heti{X_2^{d_1+d_2}X_3^{d_3}} \\ 
}$$
où on a posé
$p_1 = \coeff_{X_1^{d_1}}(P_1), \ 
q_1 = \coeff_{X_2^{d_1}}(P_1),\ 
q_2 = \coeff_{X_1^{d_1}X_2^{d_2-d_1}}(P_2), \ 
p_2 =  \coeff_{X_2^{d_2}}(P_2)$.
C'est au niveau de $q_2$ qu'intervient la contrainte $d_2 \ge d_1$.

\smallskip
Non seulement le diagramme n'est pas commutatif, mais 
les endomorphismes $B_\calM$ et $W_{\psi(\calM)}$ ne sont pas en général conjugués 
comme on le voit sur les matrices (celle de $B_\calM$ est diagonale).
Quelques détails sur l'obtention de ces matrices. 
Pour la matrice de $B_\calM$ relativement à la base monomiale $\calM$ :
par définition, la première colonne est 
$\pi_\calM\big(X_3^{d_3}P_{1}\, e_2 - X_3^{d_3} P_2\,e_1\big)$ 
et la deuxième colonne est 
$\pi_\calM\big(X_2^{d_1}P_2 \, e_3 - X_2^{d_1} P_3\, e_2\big)$.
Pour la matrice de $W_{\psi(\calM)}$  :
la première colonne est $\pi_{\psi(\calM)}\big(X_2^{d_2}  X_3^{d_3} P_{1}\big)$
et la deuxième colonne est $\pi_{\psi(\calM)}\big(X_2^{d_1} X_3^{d_3} P_2\big)$.
\end{rmq}

\subsection{Les primitives de positionnement $\Elt$ et $\Ind$ dans un ensemble totalement ordonné}
\label{soussectionEltInd}

On a pour l'instant dégagé la décomposition \og en $J$ \fg{} 
du module de Macaulay, à savoir $\Mmac_{k-1} = \bigoplus_{\#J=k-1} \Mmac[J]$.
La restriction du pont $\psi$ à chaque facteur $\Mmac[J]$ est une
injection ; l'image $\psi(\Mmac[J])$ est l'idéal engendré par les
$X_i^{d_i} X^{D(J)}$ avec $i < J$, idéal contenu dans $\Jex_k$. 
On va envisager une autre décomposition monomiale de $\Mmac_{k-1}$ pour 
laquelle la restriction du pont $\psi$ à un facteur est encore injective
mais cette fois établit un isomorphisme entre ce facteur et un idéal $\Jex_h$ avec $h \geqslant k$. 
Dit autrement, dans cette décomposition, vont figurer un certain nombre
d'exemplaires de $\Jex_h$ pour $h$ variable dans $\{k..n\}$.
Ceci revient à définir des injections monomiales $\Jex_h \to \Mmac_{k-1}$, qui, 
comme elles sont censées être des ``inverses'' du pont $\psi$, doivent être de la forme :
$$
\Jex_h \, \ni \, X^\alpha 
\ \longmapsto \ 
\frac{X^\alpha}{X^{D(J_\alpha)}}\, e_{J_\alpha}
\, \in\,  \Mmac_{k-1}, \quad
\hbox {pour un $J_\alpha$ fonction de $\alpha$, de cardinal $k-1$, à préciser.}
$$
Afin d'assurer l'exactitude du quotient, on doit imposer 
$J_\alpha \subset \DivSeq(X^\alpha)$ ; ainsi $J_\alpha$ apparaît comme une partie de 
$\DivSeq(X^\alpha)$. 
La contrainte d'être un monôme extérieur de Macaulay force l'inégalité stricte 
$\minDiv(X^\alpha) < \min(J_\alpha)$. 
Toute la subtilité consiste à imposer la position de $J_\alpha$ dans $\DivSeq(X^\alpha)$ 
(cette position est un paramètre noté $R$ dans la suite). 
Vu l'inégalité stricte ci-dessus, $J_\alpha$ se situe plutôt à
la fin de $\DivSeq(X^\alpha)$ et le positionnement va avoir lieu en
ordonnant la suite $\DivSeq(X^\alpha)$ de manière décroissante.

\bigskip

Il est maintenant temps de définir, pour un ensemble totalement
ordonné quelconque, deux primitives de positionnement Ind, Elt,
réciproques l'une de l'autre et d'en montrer quelques propriétés fondamentales. 
Les définitions qui viennent risquent de paraître déroutantes
mais le lecteur doit comprendre qu'il y a un certain prix à payer pour
obtenir le phénomène de quasi-conjugaison entre des endomorphismes
\og de Macaulay\fg{} $W_{h}(\uP)$ et des endomorphismes \og de Cayley\fg{}
de type $B_{k}(\uP)$.

\index{quasi-conjugués (endomorphismes)}%

\begin{defn}
Pour une partie finie $E$ d'un ensemble totalement ordonné, 
désignons par $\RevE$ la \textit{suite} obtenue en ordonnant l'\textit{ensemble}~$E$ 
de manière \textit{décroissante}.

\smallskip
\noindent
Pour $J \subset E$ et $R \subset \{1..\#E\} \subset \bbN$, on note :

\begin{itemize}
\item  $\Ind(J \,; E)$ l'ensemble des indices de $J$ dans la suite~$\RevE$
\item $\Elt(R \,; E)$ la partie de $E$ constituée des éléments d'indice $R$ dans la suite~$\RevE$
\end{itemize}

\smallskip
\noindent
Les primitives $\Ind$ et $\Elt$ sont duales l'une de l'autre dans le sens où 
$$
R = \Ind( J \,; E) \quad \Longleftrightarrow \quad J = \Elt(R \,; E)
$$
\end{defn}

\label {NOTA15-RevE}%
\label {NOTA15-Ind}%
\label {NOTA15-Elt}%
%
%

\begin{lem} \label{EltVersusInd}
Dans le contexte de la définition ci-dessus, on note $\preccurlyeq$ la
relation d'ordre sur les parties de même cardinal, définie par
l'inégalité terme à terme lorsque l'on ordonne lesdites parties (de
manière croissante ou de manière décroissante !).

\smallskip
\noindent
Pour une partie finie $E$, la primitive $\Ind$ possède les 3 propriétés suivantes : 

\begin{enumerate}[\rm (A)]
\item Pour toute partie finissante $F$ de $E$ contenant $J$, on a
$\Ind(J \,; E) = \Ind(J \,; F)$; finissante signifiant que $x \in F$ et
$y \geqslant x$ implique $y \in F$.

\item Soit $J \subset E' \subset E$. 
Alors $\Ind(J \,; E') \preccurlyeq \Ind(J \,; E)$.

\item Soient $J \subset E$ et $i$ un élément de l'ensemble totalement ordonné 
tel que $i < \min E$. 
Pour tout $j \in J$, 
on a 
$\Ind(J \,; E) \preccurlyeq\Ind(i\vee J \setminus j \, \,;  \, i \vee E \setminus j)$.
\end{enumerate}

\noindent
La primitive $\Elt$ possède la propriété :
$$
\big[\, E' \subset E, \quad R \subset \{1..\#E'\} \, \big]
\quad \Longrightarrow \quad
\Elt(R \,; E') \, \preccurlyeq \, \Elt(R \,; E)
$$
\end{lem}

\begin{proof}
On se permet de noter $r = \Ind( x \,; E)$ au lieu de $\{r\} = \Ind(\{x\} \,; E)$.
Commençons par une remarque élémentaire mais cruciale : 
dire que $r = \Ind(j \,; E)$, 
c'est dire que le nombre d'éléments de $E_{\geqslant j}$ est exactement $r$, 
autrement dit $\Ind(j \,; E) = \# E_{\geqslant j}$. 

Pour le point (A), il s'agit de montrer que les deux fonctions $\Ind(\sbullet \,; E)$ 
et $\Ind(\sbullet \,; F)$ sont égales sur~$J$. 
Mais puisque $F$ est finissante, on a évidemment $\#E_{\geqslant j} =\#F_{\geqslant j}$ 
donc $\Ind(j \,; E) = \Ind(j \,; F)$ et ceci est valable pour tout $j \in J$.

Pour (B), on remarque que $E'_{\geqslant j} \subset E$ donc 
$E'_{\geqslant j} \subset E_{\geqslant j}$.
En donnant un coup de $\#$, on obtient directement $\Ind(j \,; E') \leqslant \Ind(j \,; E)$ 
pour tout $j \in J$.

Pour (C), notons $J = (j_1 < j_2 < \cdots < j_m)$ et supposons que $j = j_\ell$.
Représentons $J \subset E$ et en-dessous $i \vee J \setminus j \ \subset \ i \vee E \setminus j$.
$$
\includegraphics[scale = 0.8]{AboutIndAndEltIn15.mps}     
$$
Posons $r_p = \Ind(j_p \,; E)$. Avec le schéma ci-dessus, on en déduit 
les indices de $J$ dans $E$ puis de $i \vee J \setminus j$ dans $i \vee E \setminus j$ :
$$
\begin {array} {*{17}c}
r_1 & > & r_2 & > & \cdots & > &
   r_{\ell-1} & > & r_{\ell} & > & r_{\ell+1} & > & \cdots & > & r_{m-1} & > & r_m
\\ [0.5em]
\#E & > & r_1 -1 & > & \cdots & > & 
   r_{\ell-2}-1 & > & r_{\ell-1} -1& > & r_{\ell+1} & > &\cdots & > & r_{m-1} & > & r_m
\\
\end {array}
$$
Comme $r_{p-1} > r_p$, on a $r_p \leqslant r_{p-1} - 1$.
De plus $r_1 \leqslant \#E$.
On a donc 
$\Ind(J \,; E) \preccurlyeq\Ind(i\vee J \setminus j \, \,;  \, i \vee E \setminus j)$.

\bigskip
La propriété concernant $\Elt$. 
Remarquons tout d'abord l'équivalence $j = \Elt(r \,; E) \Leftrightarrow \# E_{\geqslant j} = r$
qui est duale de  $r = \Ind(j \,; E) \Leftrightarrow \# E_{\geqslant j} = r$.
Soit $r \in R$ et notons $j = \Elt(r \,; E)$ et $j' = \Elt(r \,; E')$. 
On a $E' \subset E$ donc $E'_{\geqslant j'} \subset E_{\geqslant j'}$. 
En donnant un coup de $\#$, on a $r \leqslant \# E_{\geqslant j'}$.
Or $r = \# E_{\geqslant j}$.
On obtient donc $\# E_{\geqslant j}\, \leqslant\, \# E_{\geqslant j'}$
ce qui force $j' \leqslant j$ puisque $j$ et $j'$ sont dans $E$.
\end{proof}

\subsection{La $R$-décomposition et ses conséquences}

Nous avons conscience que la définition qui vient puisse paraître
intriguante : les objets qui interviennent ne sont pas d'une nature
algébrique courante et relèvent plutôt d'une combinatoire
(élémentaire) sur les ensembles totalement ordonnés finis.  Elle se
révèle primordiale dans la suite de cette section, et pour cette
raison, en introduction, nous en avons traité un cas particulier
auquel la lectrice pourra se reporter en cas de difficulté.  Il faut
noter que $R \subset \{1..n\}$ est vu comme un ensemble d'indices et,
en un certain sens, il va être question de situer $J \subset \{1..n\}$
par rapport à $R$ dans l'ensemble de divisibilité $\DivSeq(\XbJ)$, ce
qui est licite puisque l'on a toujours $J \subset \DivSeq(\XbJ)$.

\begin{defn}
Pour $R \subset \{1..n\}$, notons $\Mmac^R$ le sous-module monomial de
$\Mmac_{\#R}$ de base les monômes extérieurs $X^\beta e_J$ de
$\Mmac_{\#R}$ tels que $\#J = \#R$ et $\Ind \big( J \,; \DivSeq(\XbJ)
\big) = R$.  On dispose ainsi, pour chaque $k' \ge 0$, d'une
décomposition monomiale $\Mmac_{k'}=\bigoplus\limits_{\#R=k'}
\Mmac^R$, appelée {\rm \Rdecomposition}.
\end{defn}

\label {NOTA15-MmacR}%

\medskip

Il s'agit avant toute chose de \og déjouer\fg{} rapidement cette définition en
manipulant les ingrédients dans des cas simples.

\smallskip
$\rhd$
Pour $R = \emptyset$, seule partie de cardinal $0$, on a $\Mmac^{\emptyset} = \Mmac_0$,
identique à $\Jex_1$.

$\rhd$
Soit $R = \{1\}$. Que signifie l'égalité $\Ind \big(\{j\} \,;
\DivSeq(X^\beta X_j^{d_j})\big) = R$ pour un monôme de
Macaulay~$X^\beta\,e_j$?  Tout simplement que $j = \maxDiv(X^\beta
X_j^{d_j})$. Et comme $X^\beta e_j$ est de Macaulay, en notant $i =
\minDiv(X^\beta)$, on a $i < j$. En conséquence, en posant $X^\alpha = X^\beta
X_j^{d_j}$, on a $X^\alpha \in \Jex_2$ et $\maxDiv(X^\alpha) =
j$.

\smallskip
$\rhd$
Montrons l'implication $n \in R \Longrightarrow \Mmac^R = 0$, ce qui fera
\og dans la pratique \fg{} que l'on peut se restreindre à $R \subset \{1..n-1\}$.

Soit $X^\beta e_J$ un monôme de Macaulay.  On a $J \subset \{2..n\}$
car $\minDiv(X^\beta) < \min J$.  De manière duale, pour $R =
\Ind\big(J \,; \DivSeq(\XbJ)\big)$, on a $R \subset \{1..n\moins1\}$ :
ceci est dû au fait que les indices sont considérés relativement à la
suite décroissante $\DivG\big(\XbJ\big)$.  Par conséquent, si $R \not
\subset \{1..n\moins1\}$ (\idest{} si $n \in R$), alors $\Mmac^R = 0$.

\smallskip
$\rhd$
Soit $X^\beta e_J \in \Mmac^R$.  Notons $X^\alpha = X^\beta X^{D(J)}$.
Alors, par définition de l'appartenance à $\Mmac^R$, l'entier $\min(J)$
est en position $\max(R)$ dans $\DivG(X^\alpha)$.  De plus, comme
$X^\beta e_J$ est un monôme de Macaulay, on a $\minDiv(X^\beta) < \min
J$, donc $\DivSeq(X^\alpha)$ est au moins de cardinal $1 + \max(R)$.
Ainsi, en posant $h = 1+\max(R)$, on a $X^\alpha \in \Jex_h$.

\begin{exemple}[Maniement des primitives $\Ind$ et $\Elt$]

Ce qui suit peut sans doute permettre au lecteur de se familiariser
avec les primitives $\Ind$ et $\Elt$.  Partons de
$R = \{2,3,5,7\}$ ; ici $k' = \#R = 4$ et $h = 1 + \max(R) = 8$.  Fixons $n
= 13$.  L'objectif est de construire un monôme de Macaulay $X^\beta
e_J \in \Mmac_{k'}$ tel que $J = \Ind\big(R \,; \DivSeq(\XbJ)\big)$.  On a
nécessairement $\#\DivSeq(\XbJ) \geqslant h=8$, par exemple :
$$
\DivSeq(X^\beta X^{D(J)}) =
\{\ 1,2,4, \overset{(7)}{5},6, \overset{(5)}{8},9,
\overset{(3)}{10},\overset{(2)}{12}, 13\ \}
$$
Nous avons indiqué les éléments de $J$ imposés par $R$. Le
complémentaire de $\DivSeq(X^\beta X^{D(J)})$ dans $\{1..n\}
= \{1..13\}$ étant constitué des $i \in \{3,7,11\}$, pour chacun de
ces $i$, nous avons pris $\beta_i = d_i-1$ pour forcer $i \notin
\DivSeq(X^\beta X^{D(J)})$. Voici un monôme qui convient :
$$
X^\beta e_J \ = \
X_1^{d_1}\,X_2^{d_2}\,X_3^{d_3-1}X_4^{d_4}\,X_5^{\beta_5}\,X_6^{d_6}\,X_7^{d_7-1}
X_8^{\beta_8}\,
X_9^{d_9}\,X_{10}^{\beta_{10}}\,X_{11}^{d_{11}-1}X_{12}^{\beta_{12}}\,X_{13}^{d_{13}}\ 
e_{\{5,8,10,12\}}
$$
où $\beta_5$, $\beta_8$, $\beta_{10}$ et $\beta_{12}$ sont quelconques.

\medskip
\noindent
Notons $X^\alpha = X^\beta X^{D(J)}$. Puisque $X^\beta e_J$
est de Macaulay, on a $\minDiv(X^\alpha) = \minDiv(X^\beta)
< \min(J)$. En adéquation avec l'exemple,
nous obtenons le schéma suivant dans lequel $j_1 = \min(J)$.
$$
\DivSeq(X^\alpha)\ : \qquad
\vcenter{
\xymatrix @R=3pt @C=0.5cm{
        &      &      &\ar@{<->}[rrrrrr]^-{\#=\max(R)=h-1} &  &
&       &              &             &        
\\
\bullet &\circ &\circ &\blacklozenge &\circ &\blacklozenge
&\circ &\blacklozenge  &\blacklozenge  &\circ  
\\
\scriptstyle\minDiv(X^\alpha)  &      &      &j_1           &      &
&      &               &j_{k'}        
\\
}
}
$$
\label{SchemaJR}
\end{exemple}

\bigskip

La \Rdecomposition{} du module de Macaulay possède deux bonnes propriétés qui font
l'objet du lemme suivant:
d'une part la restriction du pont~$\psi$ à $\Mmac^R$ est injective 
et d'autre part l'image $\psi(\Mmac^R)$ est un idéal excédentaire \og en entier \fg{}, 
à savoir l'idéal $\Jex_{1+\max(R)}$.
Cette dernière propriété n'était pas vérifiée par $\Mmac[J]$, 
puisque $\psi(\Mmac[J])$ est l'idéal engendré par les $X_i^{d_i}X^{D(J)}$ avec $i < J$, 
ce qui ne correspond pas à un idéal excédentaire particulier.

A noter que l'égalité $\psi(\Mmac^R) = \Jex_{1+\max(R)}$ est valide pour 
$R = \emptyset$ car, avec notre convention habituelle, on a
$\max(\emptyset) = 0$ pour la partie vide de $\{1..n\}$; cette égalité valide
raconte que $\Mmac^R = \Mmac_0 = \Jex_1$.

\medskip

Dans le lemme suivant, nous allons établir que la restriction de $\psi$ à
$\Mmac^R$ est injective d'image $\Jex_{1+\max(R)}$. En conséquence:
$$
\Mmac^R = 0 \iff \Jex_{1+\max(R)} = 0 \iff 1 + \max(R) > n  \iff n \in R
$$

\begin{lem}[L'isomorphisme $\psi^R : \Mmac^R \simeq \Jex_{1+\max(R)}$ induit par le pont]
\label{psiRiso}  
\leavevmode
  
Soit $R \subset \{1..n\}$ et $h = 1+\max R$. 

\begin{enumerate}[\rm i)]
\item L'application $\psi^R$, restriction du pont $\psi$ à $\Mmac^R$, est injective,
d'image $\Jex_h$:
$$
\psi^R : \Mmac^R  \overset{\simeq}{\longrightarrow} \Jex_h
$$  
L'application réciproque est donnée par 
$$
(\psi^R)^{-1} : \begin{array}[t]{rcl}
\Jex_h & \longrightarrow & \Mmac^R \\ [0.2cm]
X^\alpha & \longmapsto & \dfrac{X^\alpha}{X^{D(J)}}e_J 
\quad \text{avec $J = \Elt\big(R \,; \DivSeq(X^\alpha) \big)$}
\end{array}
$$

\item
Soit $J \subset \{1..n\}$ avec $\#J = \#R$.  
L'image $\psi\big(\Mmac^R[J]\big)$ est le sous-module monomial de $\Jex_h$ de base
les $X^\alpha \in \Jex_h$ tels que $\Elt\big(R \,; \DivSeq(X^\alpha)\big) = J$. 

\end{enumerate}
\end {lem}

\label {NOTA15-JexR}%
\label {NOTA15-psiR}%
%
%

\begin {proof} \leavevmode

i)
Pour l'inclusion $\psi(\Mmac^R) \subset \Jex_h$, soit $X^\alpha = \XbJ \in \psi(\Mmac^R)$.
On a $\# \DivSeq(X^\alpha)_{\geqslant \min J} = \max R = h-1$ par définition de $R$.
De plus, $\DivSeq(X^\alpha)$ contient un indice $i$ strictement inférieur à $\min J$ 
à savoir $i = \minDiv(X^\beta)$. Donc $\#\DivSeq(X^\alpha) \geqslant h$.

\medskip

Dans l'autre sens, montrons qu'un $X^\alpha \in \Jex_h$ est l'image par $\psi$ 
d'un monôme extérieur $X^\beta e_J$ de $\Mmac^R$, 
c'est-à-dire vérifiant $\minDiv(X^\beta) < \min J$ et $\Ind \big( J \,; \DivSeq(\XbJ) \big) = R$.
On prend
$$
J = \Elt\big(R \,; \DivSeq(X^\alpha) \big), \qquad
X^\beta = \frac{X^\alpha}{X^{D(J)}}
$$
A gauche, le $\Elt$ est licite car $\# \DivSeq(X^\alpha) \geqslant \max R$ et à droite
le quotient est exact puisque $J \subset \DivSeq(X^\alpha)$.
Montrons que $X^\beta e_J \in \Mmac^R$, prouvant $\Jex_h \subset \psi(\Mmac^R)$
puisque $\psi(X^\beta e_J) = X^\alpha$.

\smallskip
\noindent
$\rhd$
Le monôme $X^\beta e_J$ est de Macaulay. 
En effet, $\max R = h-1$, donc ${\min J = \Elt\big(h \moins 1, \DivSeq(X^\alpha) \big)}$. 
De plus $\#\DivSeq(X^\alpha) \geqslant h$, donc $\minDiv(X^\alpha) < \min J$ puis 
$\minDiv(X^\beta) = \minDiv(X^\alpha) < \min J$.

\smallskip
\noindent
$\rhd$
Vérifions enfin que $X^\beta e_J$ est dans~$\Mmac^R$. 
Cela repose sur la dualité entre $\Ind$ et $\Elt$ :
puisque $J = \Elt\big(R \,; \DivSeq(X^\alpha) \big)$, on~a 
$R = \Ind \big( J \,; \DivSeq(X^\alpha) \big)$, donc
$R = \Ind \big( J \,; \DivSeq(\XbJ) \big)$.

\smallskip

En utilisant le fait que les primitives $\Ind$ et $\Elt$ sont
inverses l'une de l'autre, on vérifie que les correspondances
$X^\beta e_J \mapsto X^\alpha$ et $X^\alpha \mapsto X^\beta e_J$ sont
réciproques l'une de l'autre.

\medskip

ii) Résulte de la définition de $\Mmac^R[J]$ et du point i).
\end {proof}

\begin{rmqs} \label{BigRmq}
\leavevmode 

\smallskip
$\rhd$
Pour $R = \{1..n\moins1\}$, on a $\Mmac^R = \Mmac_{n-1}$ qui est
isomorphe à $\Jex_n$ en tant que $\bfA[\uX]$-module.  Donnons à
nouveau quelques détails sur cet isomorphisme.  Tout d'abord, $\Jex_n$
est l'idéal principal engendré par $X^D = X_1^{d_1} \cdots X_n^{d_n}$.
Et $\Mmac_{n-1}$ est le $\bfA[\uX]$-module engendré par $X_1^{d_1} e_2
\wedge \dots \wedge e_n$.  La correspondance bijective $\psi^R$ se
réduit donc ici à :
$$
\Mmac_{n-1} \, \ni\, 
X^\gamma X_1^{d_1} e_2 \wedge \dots \wedge e_n 
\ \longleftrightarrow \ 
X^\gamma X_1^{d_1} \cdots X_n^{d_n} \, \in \, \Jex_n
$$

\smallskip
$\rhd$
Soit $R\subset\{1..n\}$ et $h = 1+ \max R$.  On peut se demander
quelle est l'image par $(\psi^R)^{-1}$ dans $\Mmac^R$ de $X^D =
X_1^{d_1} \cdots X_n^{d_n} \in \Jex_h$.  Comme $\DivSeq(X^D) =
\{1..n\}$, cette image vaut :
$$
\dfrac{X^D}{X^{D(J)}}\, e_J 
\qquad 
\text{avec $J = \Elt \big(R \,; \{1..n\}\big)$}
$$
Mais il faut prendre garde au fait que $\Elt \big(R \,; \{1..n\}\big)$
n'est pas égal à $R$.  Cela vient du fait que $\Elt(R \,; E)$ désigne
l'ensemble des éléments d'indice $R$ dans la suite décroissante
$\RevE$, de sorte que $\Elt\big(R \,; \{1..n\}\big)$ est l'image de
$R$ par l'involution renversante $\rho$ de $\{1..n\}$ définie par
$\rho : i \mapsto n-i+1$.  Autrement dit, l'image de $X^D \in \Jex_h$
dans $\Mmac^R$ est exactement $X^{D(\overline J)} e_J$ avec $J =
\rho(R)$. On peut remarquer que pour $R = \emptyset$, on trouve bien
ce à quoi on s'attend, à savoir $X^D$ !

\smallskip
$\rhd$
Pour $R = \{r\} \subset \{1..n\moins1\}$, l'isomorphisme $\psi^R$ du
lemme \ref{psiRiso} et son inverse sont:
$$
\begin{array}[t]{rcl}
\Mmac^{\{r\}} & \longrightarrow & \Jex_{1+r} \\ [0.5em]
X^\beta e_j & \longmapsto & X^\beta X_j^{d_j}
\end{array}
\qquad \text{et} \qquad
\begin{array}[t]{rcl}
\Jex_{1+r} & \longrightarrow & \Mmac^{\{r\}} \\ [0.4em]
X^\alpha & \longmapsto & \dfrac{X^\alpha}{X_j^{d_j}} e_j
\qquad \text{avec 
$j = \Elt\big(r \,; \DivSeq(X^\alpha)\big)$} 
\end{array}
$$
En particulier, pour $r=1$, l'élément $j$ à droite vaut $\maxDiv(X^\alpha)$.
\end{rmqs}

\subsubsection*{Où l'on rapproche $\det W_{h,d}$ d'un \og sous-déterminant de Cayley\fg{} $\det B^R_{k,d}$
de $\det B_{k,d}$}

Maintenant, nous allons réaliser un mélange que nous jugeons \og explosif\fg{} :
nous allons faire intervenir les deux décompositions : celle en $R$, raffinée par
celle en $J$.
Précisément, nous allons appliquer la \Jdecomposition{} au module $\Mmac^R$, 
ce qui fournira les facteurs $\Mmac^R[J]$ où $J$ varie dans l'ensemble des
parties de $\{1..n\}$ de cardinal $\#R$. Et nous allons également faire
\og opérer\fg{} $\uP$ sur ces divers sous-modules monomiaux par l'intermédiaire
d'endomorphismes attachés à $\uP$.

\begin {defn} [L'endomorphisme induit-projeté $B^R = B^R(\uP)$ sur $\Mmac^R$]
\leavevmode


Soit $k = 1+\#R$. Rappelons encore une fois que $B_k = B_k(\uP)$ est
un endomorphisme de $\Mmac_{k-1} = \Mmac_{\#R}$. Bien évidemment,
$\Mmac^R \subset \Mmac_{\#R} = \Mmac_{k-1}$.  Par définition, $B^R$
est l'induit-projeté de $B_k$ sur~$\Mmac^R$.  Par \og sécurité\fg{},
nous noterons $B^R_{k,d}$ la composante homogène de degré $d$ de
$B^R$, l'indice $k$ ne figurant que pour des raisons esthétiques.  Dit
autrement, en présence de la construction $B^R_{k,d}$, l'indice $k$
vaut obligatoirement $1+\#R$, mais il nous arrivera de le répéter.
\end {defn}  

\label {NOTA15-BPR}%
%
%

\begin{prop} \label{BRquasi-conjWh}
Soient $R \subset \{1..n\}$ et $h = 1+\max R$.
Dans l'énoncé et la preuve, on adopte la convention
suivante: $J \subset \{1..n\}$
vérifie $\#J = \#R$ et on lui associe $\Jex^R[J] \overset{\rm def}{=}
\psi\big(\Mmac^R[J]\big) \subset \Jex_h$.

\begin{enumerate}[\rm i)]
\item 
L'endomorphisme $B^R$ est triangulaire relativement à la décomposition
$\Mmac^R = \bigoplus\limits_J \Mmac^R[J]$.

L'endomorphisme $W_h$ est triangulaire relativement à la décomposition
$\Jex_h = \bigoplus\limits_J \Jex^R[J]$.

Précision: pour $J$ fixé, on dispose des inclusions suivantes
avec un renversement de l'ordre $\preccurlyeq$ dans la somme directe de droite
par rapport à celle de gauche:
$$
B^R\big(\,\Mmac^R[J]\,\big) 
\ \subset \ 
\bigoplus_{J' \preccurlyeq J} \Mmac^R[J']
\qquad 
\text{et }
\qquad 
W_h \big(\JexRJ \big) 
\ \subset \ 
\bigoplus_{J' \succcurlyeq J} \Jex^R[J']
$$

\item
Les endomorphismes $B^R$ et $W_h$ sont quasi-conjugués relativement à la \Jdecomposition.
Autrement dit, pour tout $J$ avec $\#J=\#R$, 
on dispose du diagramme commutatif de $\bfA$-modules :
$$
\def \Binduitproj{\begin{array}{c}
     \hbox {induit-proj. de $B^R$} \\[1mm] \hbox {noté $B^R[J]$} \end {array}}
\def \Winduitproj{\begin{array}{c}
     \hbox {induit-proj. de $W_h$} \\[1mm] \hbox {noté $W_h^R[J]$} \end {array}}
\xymatrix @R = 2cm @C = 4cm @M=0.4pc{
\Mmac^R[J] \ar[d]|-{\Binduitproj} \ar[r]^{\psi^R_J}_{\simeq}
& 
\Jex^R[J]
\ar[d]|-{\Winduitproj}
\\
\Mmac^R[J] \ar[r]^{\psi^R_J}_{\simeq} & 
\Jex^R[J]
\\
}
$$

\item 
En particulier, pour $k = 1 + \#R$, on a $\det B_{k,d}^R = \det W_{h,d}$. Ou encore,
en exprimant cette égalité directement en fonction de $R$:
$$
\det B_{1+\#R,d}^R = \det W_{1+\max(R),d}
$$
\end{enumerate}
\end{prop}

\index{quasi-conjugués (endomorphismes)}%
\label {NOTA15-JexRJ}%
%
%

\begin{proof} \leavevmode
  
i) En utilisant~\ref{JDecompositionEndo}-ii) avec $\calM = \Mmac^R$, on obtient 
l'aspect triangulaire de $B^R$.

\smallskip
Occupons-nous de $W_h$.
Soit $X^\alpha \in \JexRJ$ c'est-à-dire $X^\alpha \in \Jex_h$ et 
$\Elt\big(R \,; \DivSeq(X^\alpha)\big) = J$. 
Ci-dessous, on va profiter du fait que l'on a également
$J = \Elt\big(R \,; \DivSeq(X^\alpha)_{\geqslant \min J} \big)$. 
Rappelons que dans ces conditions, on a 
$i := \minDiv(X^\alpha) < \min J$ (voir le schéma page~\pageref{SchemaJR}).
L'image $W_h(X^\alpha)$ est une $\bfA$-combinaison linéaire de monômes $X^{\alpha'}$ de la forme 
$$
X^{\alpha'} \ = \ 
\dfrac{X^\alpha}{X_i^{d_i}} X^\gamma
\qquad \text{où $X^\gamma$ est un monôme de $P_i$}
$$
Chaque $X^{\alpha'}$ est dans $\Jex^R[J']$ où $J' = \Elt\big(R \,; \DivSeq(X^{\alpha'})\big)$.
Par définition de $X^{\alpha'}$, 
on a ${\DivSeq(X^\alpha) \setminus i \subset \DivSeq(X^{\alpha'})}$.
Comme $i < \min J$, on en déduit que 
$\DivSeq(X^\alpha)_{\geqslant \min J} \subset \DivSeq(X^{\alpha'})$.
On donne un coup de $\Elt(R \,; \sbullet)$ et avec le lemme~\ref{EltVersusInd}, 
on obtient $J \preccurlyeq J'$.

\medskip
ii) Utiliser~\ref{JDecompositionEndo}-iii) avec $\calM = \Mmac^R$.

\medskip

iii)
Bien entendu, l'esprit des \og stratégies triangulaires\fg{} souffle ici (cf. la section
\ref{subsectionDecompositionTriangulaire} et en particulier
la proposition~\ref{EndosQuasiConjugues}). On obtient, en utilisant
le point i) ci-dessus:
$$
\det B^R_{k,d} = \prod_{\#J = \#R} \det B_{k,d}^R[J]
\qquad\qquad
\det W_{h,d} = \prod_{\#J = \#R} \det W_{h,d}^R[J]
$$
Et d'après le point ii), on a $\det B_{k,d}^R[J] = \det W_{h,d}^R[J]$,
ce qui permet de conclure. 
\end{proof}

\subsubsection*{Comment illustrer cette proposition \ref{BRquasi-conjWh}?}

Concentrons-nous du côté de $\Jex_h$ en fixant un $R$ tel que $h =
1+\max(R)$ et oublions (provisoirement) le reste.  La décomposition monomiale de
l'idéal~$\Jex_h$ en la somme directe des $\big(\Jex^R[J]\big)_{\#J=\#R}$
peut être racontée de manière plus terre à terre, disons intrinsèque à $\bfA[\uX]$,
sans référence aux sous-modules de Macaulay $\Mmac_\sbullet$.

\medskip

Cette description s'appuie sur le fait que $\Jex^R[J]$ est le
sous-module monomial de $\Jex_h$ de base les $X^\alpha \in \Jex_h$
tels que $\Elt\big(R \,; \DivSeq(X^\alpha)\big) = J$. On réalise ainsi
une partition des monômes $X^\alpha \in \Jex_h$ \og par égalité des $J
= \Elt\big(R \,; \DivSeq(X^\alpha)\big)$\fg{} et ce phénomène de
partition des monômes nous est familier.

\medskip

Quels paramètres choisir pour \emph {montrer} des exemples pertinents?  Le cas
$h=2$ force $R=\{1\}$ et la partition associée des monômes de $\Jex_2$
n'est autre que celle obtenue par égalité des~$\maxDiv$. 
Nous en avons parlé dans un chapitre
antérieur (cf. la section \ref{SousSectionMaxDivDecomposition} dans
laquelle nous avons utilisé, pour un sous-module monomial $\calM$ de
$\bfA[\uX]$, la notation $\calM\double{j}$ pour désigner le
sous-module monomial de $\calM$ de base les monômes de $\maxDiv = j$).

\smallskip

Nous fournissons un tel exemple avec $h=2$ un peu plus loin (cf. la
page~\pageref{ExempleD312R1}) mais nous préférons commencer par le cas
qui suit i.e. par le choix de \boxed{h=3} donc $\max(R) = 2$ i.e. $R =
\{2\}$ ou $R = \{1,2\}$. Après un certain nombre de tâtonnements, nous
nous sommes fixés sur:

\subsubsection*{Un premier exemple $n = 4$, $D = (1,2,1,2)$, $d = 5$ pour illustrer
la proposition \ref{BRquasi-conjWh}}


Dans cet exemple, $\dim \Jex_{h,d} = 10$.  Nous ne montrons
pas le système $\uP$ mais disons seulement que les noms des
coefficients de $P_1, P_2, P_3, P_4$ sont dans l'ordre $a_\sbullet,
b_\sbullet, c_\sbullet, d_\sbullet$. Dans les matrices, $\overline
{c_1}$ désigne $-c_1$. Enfin, pour des raisons typographiques nous
utilisons $x_i$ au lieu de $X_i$.

\medskip
$\bullet$
Commençons par $R = \{2\}$. La partition associée est constituée de
2 morceaux de même taille 5, un pour $J=\{2\}$, l'autre pour $J=\{3\}$:
$$
\overbrace{
x_{1}^{2}x_{2}^{2}x_{3},\ x_{1}x_{2}^{3}x_{3},\ x_{1}x_{2}^{2}x_{3}^{2},\ x_{1}x_{2}^{2}x_{3}x_{4},\  x_{1}x_{2}^{2}x_{4}^{2}}
^{J=\{2\}}
\qquad
\overbrace{
x_{1}^{2}x_{3}x_{4}^{2},\ x_{1}x_{2}x_{3}x_{4}^{2},\ x_{1}x_{3}^{2}x_{4}^{2},\
x_{1}x_{3}x_{4}^{3},\ x_{2}^{2}x_{3}x_{4}^{2}}
^{J=\{3\}}
$$
Ce qui fournit pour $W_{h,d}(\uP)$ (matrice de droite) une structure triangulaire \emph{inférieure}  en 2 blocs.
Ceci est en cohérence avec le point i) de la proposition \ref{BRquasi-conjWh} et la mention 
$W_h \big(\JexRJ \big) \subset \bigoplus_{J' \succcurlyeq J} \Jex^R[J']$.
$$
\def\ov{\overline}
\left[
\begin{array}{*{5}{c}|*{5}{c}}
a_{1}& .& .& .& .& .& .& .& .& . \\ 
a_{2}& a_{1}& .& .& .& .& .& .& .& . \\ 
a_{3}& .& a_{1}& .& .& .& .& .& .& . \\ 
a_{4}& .& .& a_{1}& .& .& .& .& .& . \\ 
.& .& .& .& a_{1}& .& .& .& .& \ov{c_{1}} \\ 
\cline {1-10}
.& .& .& .& .& a_{1}& .& .& .& b_{1} \\ 
.& .& .& .& .& a_{2}& a_{1}& .& .& b_{2} \\ 
.& .& .& .& .& a_{3}& .& a_{1}& .& b_{3} \\ 
.& .& .& .& .& a_{4}& .& .& a_{1}& b_{4} \\ 
.& .& .& .& .& .& a_{2}& .& .& b_{5} \\ 
\end{array}
\right]
\qquad
\left[
\begin{array}{*{5}{c}|*{5}{c}}
a_{1}& .& .& .& .& .& .& .& .& . \\ 
a_{2}& a_{1}& .& .& .& .& .& .& .& . \\ 
a_{3}& .& a_{1}& .& .& .& .& .& .& . \\ 
a_{4}& .& .& a_{1}& .& .& .& .& .& . \\ 
.& .& .& .& a_{1}& .& .& .& .& . \\ 
\cline {1-10}
.& .& .& .& .& a_{1}& .& .& .& b_{1} \\ 
.& .& .& .& .& a_{2}& a_{1}& .& .& b_{2} \\ 
.& .& .& .& .& a_{3}& .& a_{1}& .& b_{3} \\ 
.& .& .& .& .& a_{4}& .& .& a_{1}& b_{4} \\ 
.& .& .& a_{4}& a_{3}& .& a_{2}& .& .& b_{5} \\ 
\end{array}
\right]
$$
Faisons maintenant intervenir $k= 1+\#R = 2$ et le pont $\psi^R : \Mmac^R \overset{\simeq}{\longmapsto}
\Jex_h=\Jex_3$ où $\Mmac^R \subset \Mmac_{k-1} = \Mmac_1$. Par définition, $\psi^R$
est la restriction de $\psi$ et $\psi$ sur $\rmK_1$ est
le $\bfA[\uX]$-morphisme réalisant $e_j \mapsto X_j^{d_j}$:
$$
\psi_{|\rmK_1} : \qquad e_1 \mapsto x_1, \qquad e_2 \mapsto x_2^2, \qquad e_3 \mapsto x_3,\qquad e_4 \mapsto x_4^2
$$
La base monomiale de la composante homogène de degré $d=5$ de $\Mmac^R$ est l'image
par $(\psi^R)^{-1}$ de la base monomiale de $\Jex_{h,d}$:
$$
x_{1}^{2}x_{3}\,e_{2},\ x_{1}x_{2}x_{3}\,e_{2},\ x_{1}x_{3}^{2}\,e_{2},\ x_{1}x_{3}x_{4}\,e_{2},
x_{1}x_{4}^{2}\,e_{2},
\qquad
x_{1}^{2}x_{4}^{2}\,e_{3},\ x_{1}x_{2}x_{4}^{2}\,e_{3},\ 
x_{1}x_{3}x_{4}^{2}\,e_{3},\ x_{1}x_{4}^{3}\,e_{3},\ x_{2}^{2}x_{4}^{2}\,e_{3} 
$$
Ceci conduit à la matrice de gauche $B^R_{k,d}(\uP)$, qui est triangulaire \emph{supérieure} en
2 blocs, en accord avec le point i) de la proposition \ref{BRquasi-conjWh}.
Et comme annoncé dans cette proposition, les blocs diagonaux des 2 matrices sont identiques.
On rappelle que $W_{h,d}$ ne dépend pas des $h-1 \overset{\rm ici}{=} 2$ derniers polynômes.
Par contre, $B^R_{k,d}(\uP)$ dépend de $P_3$ (mais pas son déterminant).

\medskip
$\bullet$
Passons à $R = \{1,2\}$. Il y a trois blocs, pour $J =
\{2,3\}$, $\{2,4\}$, $\{3,4\}$, de tailles respectives $4,1,5$.
$$
\overbrace {
  x_{1}^{2}x_{2}^{2}x_{3},\ x_{1}x_{2}^{3}x_{3},\ x_{1}x_{2}^{2}x_{3}^{2},\ x_{1}x_{2}^{2}x_{3}x_{4}
}^{J = \{2,3\}}  
\qquad
\overbrace {
  x_{1}x_{2}^{2}x_{4}^{2}
}^{J = \{2,4\}}  
\qquad
\overbrace {
x_{1}^{2}x_{3}x_{4}^{2},\ x_{1}x_{2}x_{3}x_{4}^{2},\ x_{1}x_{3}^{2}x_{4}^{2},\
x_{1}x_{3}x_{4}^{3},\ x_{2}^{2}x_{3}x_{4}^{2}
}^{J = \{3,4\}}  
$$
Evidemment, \emph {l'endomorphisme} $W_{h,d}$ n'a pas changé entre temps. Ce qui a changé, c'est
la partition de~$\Jex_{h,d}$. Mais, coïncidence, il se trouve que les 2 bases monomiales de
$\Jex_{h,d}$ sont les mêmes, ce qui explique que les matrices de $W_{h,d}$ sont identiques.
$$
\def\ov{\overline}
\left[
\begin{array}{*{4}{c}|*{1}{c}|*{5}{c}}
a_{1}& .& .& .& .& .& .& .& .& d_{1} \\ 
a_{2}& a_{1}& .& .& .& .& .& .& .& d_{2} \\ 
a_{3}& .& a_{1}& .& .& .& .& .& .& d_{3} \\ 
a_{4}& .& .& a_{1}& .& .& .& .& .& d_{4} \\ 
\cline {1-10}
.& .& .& .& a_{1}& .& .& .& .& \ov{c_{1}} \\ 
\cline {1-10}
.& .& .& .& .& a_{1}& .& .& .& b_{1} \\ 
.& .& .& .& .& a_{2}& a_{1}& .& .& b_{2} \\ 
.& .& .& .& .& a_{3}& .& a_{1}& .& b_{3} \\ 
.& .& .& .& .& a_{4}& .& .& a_{1}& b_{4} \\ 
.& .& .& .& .& .& a_{2}& .& .& b_{5} \\ 
\end{array}
\right]
\qquad
\left[
\begin{array}{*{4}{c}|*{1}{c}|*{5}{c}}
a_{1}& .& .& .& .& .& .& .& .& . \\ 
a_{2}& a_{1}& .& .& .& .& .& .& .& . \\ 
a_{3}& .& a_{1}& .& .& .& .& .& .& . \\ 
a_{4}& .& .& a_{1}& .& .& .& .& .& . \\ 
\cline {1-10}
.& .& .& .& a_{1}& .& .& .& .& . \\ 
\cline {1-10}
.& .& .& .& .& a_{1}& .& .& .& b_{1} \\ 
.& .& .& .& .& a_{2}& a_{1}& .& .& b_{2} \\ 
.& .& .& .& .& a_{3}& .& a_{1}& .& b_{3} \\ 
.& .& .& .& .& a_{4}& .& .& a_{1}& b_{4} \\ 
.& .& .& a_{4}& a_{3}& .& a_{2}& .& .& b_{5} \\ 
\end{array}
\right]
$$
Maintenant, c'est $k=1+\#R=3$ qui entre en scène et
le pont $\psi^R : \Mmac^R \overset{\simeq}{\longmapsto}
\Jex_h$ où $\Mmac^R \subset \Mmac_{k-1} = \Mmac_2$,
restriction de $\psi : \rmK_2 \to \bfA[\uX]$:
$$
\psi_{|\rmK_2} : e_{ij} \mapsto x_i^{d_i} x_j^{d_j} 
\qquad \text{en particulier} \qquad
e_{23} \mapsto x_2^2x_3, \qquad  e_{24} \mapsto x_2^2x_4^2, \qquad  e_{34} \mapsto x_3x_4^2
$$
La base monomiale de la composante homogène $(\Mmac^R)_{d=5}$ est l'image
par $(\psi^R)^{-1}$ de la base monomiale de $\Jex_{h,d}$:
$$
x_{1}^{2}\,e_{23},\ x_{1}x_{2}\,e_{23},\ x_{1}x_{3}\,e_{23},\ x_{1}x_{4}\,e_{23},\quad
x_{1}\,e_{24},\quad 
x_{1}^{2}\,e_{34},\ x_{1}x_{2}\,e_{34},\ x_{1}x_{3}\,e_{34},\ x_{1}x_{4}\,e_{34},\ x_{2}^{2}\,e_{34}
$$
On obtient ainsi la matrice $B^R_{k,d}(\uP)$ de gauche qui dépend des 4 polynômes.
Et encore une fois, on peut constater de visu les miracles triangulaires de la
proposition \ref{BRquasi-conjWh}. Bien entendu, le degré $d=5$ n'a aucun
rôle particulier mais il fallait bien en choisir un pour l'illustration.


\subsubsection{Un second exemple $D=(3,1,2)$, $R= \{1\}$, $d=5$, $h=2$
pour illustrer $\det B_{k,d}^R = \det W_{h,d}$ dans la proposition~\ref{BRquasi-conjWh}}

Cet exemple a été écrit avant le précédent avec l'alibi suivant:
que se passe-t-il lorsque nous permutons les degrés d'un format?
Fixons $k=2$ donc $\#R=k-1=1$ i.e. $R \in \big \{  \{1\}, \{2\}, \cdots, \{n-1\} \big\}$.
Prenons $R = \{1\}$, donc $h = 1 + \max(R) = 2$. Dans ce cas, 
on a $\Elt(R \,; E) = \{\max E\}$ pour tout $E$.
L'isomorphisme $\psi^R : \Mmac^{\{1\}} \to \Jex_2$ et son inverse ont été explicités en fin 
de la remarque~\ref{BigRmq}:
$$
\begin{array}[t]{rcl}
\Mmac^{\{1\}} & \longrightarrow & \Jex_{2} \\ [0.5em]
X^\beta e_j & \longmapsto & X^\beta X_j^{d_j}
\end{array}
\qquad \text{et} \qquad
\begin{array}[t]{rcl}
\Jex_{2} & \longrightarrow & \Mmac^{\{1\}} \\ [0.4em]
X^\alpha & \longmapsto & \dfrac{X^\alpha}{X_j^{d_j}} e_j
\qquad \text{avec $j = \maxDiv(X^\alpha)$} 
\end{array}
$$
Choisissons maintenant $n=3$, $D=(3,1,2)$ et $d=5$. 
Explicitons la correspondance entre les monômes de $\Mmac_{1,5}^{\{1\}}$ et $\Jex_{2,5}$.
Ces monômes sont au nombre de $10$. Voici pour commencer ceux de $\Mmac_{1,5}^{\{1\}}$ que l'on 
regroupe par égalité de $J$ de manière à pouvoir observer la structure triangulaire de $B_{2,5}^{\{1\}}$
$$
X^4e_2,\   X^3Ye_2,\   X^3Ze_2,\
\qquad
X^3e_3,\   X^2Ye_3,\   XY^2e_3,\   XYZe_3,\    Y^3e_3,\   Y^2Ze_3,\  YZ^2e_3
$$
Et leurs correspondants par $\psi^R$ dans $\Jex_{2,5}$ : 
$$
X^4Y,\  X^3Y^2,\  X^3YZ,
\qquad
X^3Z^2,\  X^2YZ^2,\  XY^2Z^2,\  XYZ^3,\   Y^3Z^2,\  Y^2Z^3,\  YZ^4
$$
On va utiliser le système $\uP$ de format $D = (3,1,2)$ :
$$
\begin{tabular}{rcp{14cm}} 
$P_{1}$ & $=$ & $a_{1}X^{3} + a_{2}X^{2}Y + a_{3}X^{2}Z + a_{4}XY^{2} + a_{5}XYZ + a_{6}XZ^{2} + a_{7}Y^{3} + a_{8}Y^{2}Z + a_{9}YZ^{2} + a_{10}Z^{3}$\\ [0.1cm] 
$P_{2}$ & $=$ & $b_{1}X + b_{2}Y + b_{3}Z$\\ [0.1cm] 
$P_{3}$ & $=$ & $c_{1}X^{2} + c_{2}XY + c_{3}XZ + c_{4}Y^{2} + c_{5}YZ + c_{6}Z^{2}$\\ [0.1cm] 
\end{tabular} 
$$

Ci-dessous, voici les matrices des endomorphismes $B_{k,d}^R$ et $W_{h,d}$ dans les bases 
monomiales ci-dessus, dans l'ordre indiqué, avec $\overline {c_i} = -c_i$.
On y voit que les blocs carrés diagonaux 
sont les mêmes, qu'à gauche la matrice est triangulaire supérieure par blocs, et à droite
triangulaire inférieure par blocs, comme prévu par la proposition~\ref{BRquasi-conjWh}
\noindent
\label{ExempleD312R1}
$$
\def\ov{\overline}
\left[ 
\begin{array}{*{3}{c}|*{7}{c}}
a_{1} & . & . & . & \ov{c_{1}} & . & . & . & . & . \\ 
a_{2} & a_{1} & . & . & \ov{c_{2}} & \ov{c_{1}} & . & . & . & . \\ 
a_{3} & . & a_{1} & . & \ov{c_{3}} & . & \ov{c_{1}} & . & . & . \\ 
\cline {1-10}
. & . & . & a_{1} & b_{1} & . & . & . & . & . \\ 
. & . & . & a_{2} & b_{2} & b_{1} & . & . & . & . \\ 
. & . & . & a_{4} & . & b_{2} & . & b_{1} & . & . \\ 
. & . & . & a_{5} & . & b_{3} & b_{2} & . & b_{1} & . \\ 
. & . & . & a_{7} & . & . & . & b_{2} & . & . \\ 
. & . & . & a_{8} & . & . & . & b_{3} & b_{2} & . \\ 
. & . & . & a_{9} & . & . & . & . & b_{3} & b_{2} \\ 
\end{array}
\right]
\quad 
\left[ 
\begin{array}{*{3}{c}|*{7}{c}}
a_{1} & . & . & . & . & . & . & . & . & . \\ 
a_{2} & a_{1} & . & . & . & . & . & . & . & . \\ 
a_{3} & . & a_{1} & . & . & . & . & . & . & . \\ 
\cline {1-10}
. & . & . & a_{1} & b_{1} & . & . & . & . & . \\ 
a_{6} & . & a_{3} & a_{2} & b_{2} & b_{1} & . & . & . & . \\ 
a_{9} & a_{6} & a_{5} & a_{4} & . & b_{2} & . & b_{1} & . & . \\ 
a_{10} & . & a_{6} & a_{5} & . & b_{3} & b_{2} & . & b_{1} & . \\ 
. & a_{9} & a_{8} & a_{7} & . & . & . & b_{2} & . & . \\ 
. & a_{10} & a_{9} & a_{8} & . & . & . & b_{3} & b_{2} & . \\ 
. & . & a_{10} & a_{9} & . & . & . & . & b_{3} & b_{2} \\ 
\end{array}
\right]
$$
Précisément,
les blocs diagonaux de la matrice de gauche sont les 
matrices des $B_{k,d}^{R}[J]$.
Le bloc NO correspond à $J= \{2\}$ et le bloc SE correspond à $J = \{3\}$.
Du côté de $W_{h,d}$, les blocs diagonaux sont indexés par les monômes $X^\alpha$
de même $J = \Elt\big(R \,; \DivSeq(X^\alpha)\big)$ (ici de même $\maxDiv$ car $R = \{1\}$).

\bigskip

$\sbullet$
Pour préparer un tel exemple, il est bon de contrôler le nombre de
monômes dans $\Jex_{2,d}$. Nous avons déterminé (cf. compléments après
la preuve de la proposition~\ref{J1iJ2iSeries}) la série de $\Jex_2$
en tant que fraction rationnelle. Pour $n=3$, $D= (d_1,d_2,d_3)$, avec
les notations du résultat référencé:
$$
\sum_{d\ge 0} \dim \Jex_{2,d}\,t^d = \dfrac{\mathscr N_{1,2}(t) + \mathscr N_{2,2}(t)}{(1-t)^3}
$$
où
$$
\left\{
\begin {array} {rcl}
\mathscr N_{1,2}(t) &=& t^{d_1}\big(1 - (1-t^{d_2})(1-t^{d_3})\big) =
t^{d_1+d_2} + t^{d_1+d_3} - t^{d_1+d_2+d_3}  \\
\mathscr N_{2,2}(t) &=& t^{d_2} (1-t^{d_1}) \big(1 - (1-t^{d_3})\big) =
t^{d_2+d_3}  - t^{d_1+d_2+d_3}  \\
\end {array}
\right.
$$
Ainsi, la série de $\Jex_2$ est la fraction rationnelle :
$$
\sum_{d\ge 0} \dim \Jex_{2,d}\,t^d = \dfrac{t^{d_1+d_2} + t^{d_1+d_3} + t^{d_2+d_3} - 2t^{d_1+d_2+d_3}}{(1-t)^3}
$$
ce qui donne pour $D = (3,1,2)$ :
$$
\dfrac{t^3 (1 + t + t^2 - 2t^3)}{(1-t)^3} \ =\ 
t^3 + 4t^4 + 10t^5 + 17t^6 + 25t^7 + 34t^8 + 44t^9 + 55t^{10} + 67t^{11} + \cdots
$$
Le terme $10t^5$ signifie donc que la base monomiale de $\Jex_{2,5}$ est de cardinal 10.

\bigskip

$\sbullet$
Que se passerait-il si on permutait $D = (3,1,2)$ en $D=(1,2,3)$ ?  
Le nombre de monômes de $\Jex_{2,5}$ est inchangé puisque de manière générale l'idéal
$\Jex_h$ est symétrique en $(d_1, \ldots, d_n)$. En revanche, ce qui change c'est la combinatoire 
au niveau des monônes extérieurs de $\Mmac^R$. Voici ces $10$ monômes extérieurs toujours 
rangés par égalité des $J$ : 
$$
X^3e_2,\  X^2Ye_2,\  X^2Ze_2,\  XY^2e_2,\  XYZe_2,\  XZ^2e_2,\qquad
X^2e_3,\   XYe_3,\    XZe_3,\     Y^2e_3
$$
On remarquera que les monômes de $\Jex_{2,5}$, images des monômes extérieurs ci-dessus par
l'isomorphisme monomial~$\psi^R$, ne sont pas rangés comme dans l'exemple précédent $D = (3,1,2)$ :
$$
X^3Y^2,\  X^2Y^3,\  X^2Y^2Z,\  XY^4,\  XY^3Z,\  XY^2Z^2,\qquad
X^2Z^3,\  XYZ^3,\  XZ^4,\  Y^2Z^3
$$
Pour le système $\uP$ de format $D=(1,2,3)$ : 
$$
\begin{tabular}{rcp{14cm}} 
$P_{1}$ & $=$ & $a_{1}X + a_{2}Y + a_{3}Z$\\ [0.1cm] 
$P_{2}$ & $=$ & $b_{1}X^{2} + b_{2}XY + b_{3}XZ + b_{4}Y^{2} + b_{5}YZ + b_{6}Z^{2}$\\ [0.1cm] 
$P_{3}$ & $=$ & $c_{1}X^{3} + c_{2}X^{2}Y + c_{3}X^{2}Z + c_{4}XY^{2} + c_{5}XYZ + c_{6}XZ^{2} + c_{7}Y^{3} + c_{8}Y^{2}Z + c_{9}YZ^{2} + c_{10}Z^{3}$\\ [0.1cm] 
\end{tabular} 
$$
Voici la matrice de $B_{k,d}^R$ à gauche et celle de $W_{h,d}$ à droite, relativement 
aux bases monomiales ci-dessus :
$$
\def\ov{\overline}
\left[ 
\begin{array}{*{6}{c}|*{4}{c}}
a_{1} & . & . & . & . & . & . & . & . & \ov{c_{1}} \\ 
a_{2} & a_{1} & . & . & . & . & . & . & . & \ov{c_{2}} \\ 
a_{3} & . & a_{1} & . & . & . & . & . & . & \ov{c_{3}} \\ 
. & a_{2} & . & a_{1} & . & . & . & . & . & \ov{c_{4}} \\ 
. & a_{3} & a_{2} & . & a_{1} & . & . & . & . & \ov{c_{5}} \\ 
. & . & a_{3} & . & . & a_{1} & . & . & . & \ov{c_{6}} \\ 
\cline {1-10}
. & . & . & . & . & . & a_{1} & . & . & b_{1} \\ 
. & . & . & . & . & . & a_{2} & a_{1} & . & b_{2} \\ 
. & . & . & . & . & . & a_{3} & . & a_{1} & b_{3} \\ 
. & . & . & . & . & . & . & a_{2} & . & b_{4} \\ 
\end{array}
\right]
\qquad
\left[ 
\begin{array}{*{6}{c}|*{4}{c}}
a_{1} & . & . & . & . & . & . & . & . & . \\ 
a_{2} & a_{1} & . & . & . & . & . & . & . & . \\ 
a_{3} & . & a_{1} & . & . & . & . & . & . & . \\ 
. & a_{2} & . & a_{1} & . & . & . & . & . & . \\ 
. & a_{3} & a_{2} & . & a_{1} & . & . & . & . & . \\ 
. & . & a_{3} & . & . & a_{1} & . & . & . & . \\ 
\cline {1-10}
. & . & . & . & . & . & a_{1} & . & . & b_{1} \\ 
. & . & . & . & . & . & a_{2} & a_{1} & . & b_{2} \\ 
. & . & . & . & . & . & a_{3} & . & a_{1} & b_{3} \\ 
. & . & . & . & . & a_{3} & . & a_{2} & . & b_{4} \\ 
\end{array}
\right]
$$
Le miracle triangulaire des blocs a toujours lieu\dots

\subsection{Décomposition du sous-module de Macaulay en idéaux excédentaires}

En accord avec la détermination du poids en $P_i$ de $\det
B_{k,d}(\uP)$ (cf. la proposition \ref{BkdPoidsPiSerie} et aux
alentours), nous avons défini un sous-$\bfA$-module monomial de
$\Mmac_{k-1}$, noté $\Mmac_{k-1}^{(i)}$:
$$
\Mmac_{k-1}^{(i)} = \text{sous-module de base les 
$X^\beta e_J \in \Mmac_{k-1}$ tels que $\minDiv(X^\beta) = i$}
$$
La vocation de ce sous-module monomial est de vérifier:
$$
\begin {array} {rcl}
\poids_{P_i}(\det B_{k,d}) &=& \dim \Mmac^{(i)}_{k-1,d} 
\\
&=& \#\big\{X^\beta e_J \in \Mmac_{k-1,d} \ \big | \ \minDiv(X^\beta) = i \big\}
\\
&=& 
\#\big\{ X^\alpha e_I \in \Smac_{k,d} \ \big | \ \min(I) = i \big\}
\\
\end {array}
$$

\begin{prop}\label{DecompositionMkminus1ParJh}
Pour $k \geqslant 1$, le sous-module de Macaulay $\Mmac_{k-1}$ est monomialement
isomorphe à la somme directe d'idéaux $\Jex_h$ pour $h \geqslant k$ :
$$
\Mmac_{k-1} \ \simeq \ 
\bigoplus\limits_{h = k}^n  \ \Jex_h^{\oplus e_{k,h}}
\qquad 
\text{où \  $e_{k,h} = \binom{h-2}{h-k}$}
$$
Pour se souvenir de l'ordre des indices dans $e_{\sbullet,\sbullet}$:
le second indice est moralement plus grand que le premier et, dans les
formules, le premier est fixe, le second variable.

\noindent
Par exemple, 
$$
\Mmac_1 \ \simeq \
\Jex_2 \oplus \cdots \oplus \Jex_n 
\qquad 
\Mmac_2 \ \simeq \ 
\Jex_3 \oplus \Jex_4^{\oplus 2} \oplus  \cdots \oplus \Jex_n^{\oplus(n-2)} 
\qquad 
\Mmac_{n-2} \ \simeq \ \Jex_{n-1} \oplus \Jex_n^{\oplus(n-2)} 
\qquad 
\Mmac_{n-1} \, \simeq \, \Jex_n
$$
De plus, pour tout $i \in \{1..n\}$ et tout $d$, on a l'égalité des dimensions des $\bfA$-modules :
$$
\dim \Mmac_{k-1,d}^{(i)} \ = \ 
\sum\limits_{h = k}^n  \ e_{k,h} \, \dim \Jex_{h,d}^{(i)} 
\qquad \qquad 
\dim \Jex_{h,d}^{(i)} \ = \ 
\sum_{k=h}^n \ (-1)^{k-h} e_{h,k} \dim \Mmac_{k-1,d}^{(i)} 
$$
\end{prop}

\begin{proof}
Rappelons que l'on a 
$$
\Mmac_{k-1} = 
\bigoplus_{\#R = k-1} \Mmac^R
\qquad \text{et, \ pour $R \subset \{1..n\}$,} \quad 
\Mmac^R \ \simeq \ \Jex_h \quad \text{avec $h = 1+\max R$}
$$
Pour $R \not \subset \{1..n\moins1\}$, on a $\Mmac^R = 0$,
ce qui est dû au fait que $n\in R$ d'où $h = n+1$ donc $\Jex_h = 0$.
Pour les parties $R \subset \{1..n\moins1\}$, on est invité 
à les ranger suivant la valeur de leur maximum.
Or pour un entier $h \geqslant k$ fixé, 
il y a $\binom{h-2}{k-2} = \binom{h-2}{h-k}$ parties $R \subset \{1..n\moins1\}$
de cardinal $k-1$ et de maximum $h-1$.
De manière pédante, cela se traduit par la correspondance :
$$
\begin{array}{rcl}
\calP_{k-1}\big(\{1..n\moins1\}\big) & \longrightarrow & 
\bigvee\limits_{h=k}^n \ \Big( \calP_{k-2}\big(\{1..h\moins2\}\big) \times \{h\} \Big)\\ [1em]
R & \longmapsto &  \big(R \setminus \max R,\ \, 1+\max R\big) \\ [0.5em]
Q \vee \{h \moins1\} & \longleftmapsto & (Q,h)
\end{array}
$$
Ainsi
$$
\Mmac_{k-1} \ = \ 
\bigoplus_{h = k}^n \bigoplus_{\#Q = k-2 \atop Q \subset \{1..h-2\}} 
\Mmac^{Q \vee\{h-1\}}
\ \ \simeq \ \ 
\bigoplus\limits_{h = k}^n  \ \Jex_h^{\oplus \binom{h-2}{h-k}}
$$
Prouvons les égalités sur les dimensions.
L'isomorphisme $\psi^R : \Mmac^R \simeq \Jex_h$ explicité en~\ref{psiRiso} 
conserve $\minDiv$ et le degré.
Ainsi 
$\Mmac_{k-1,d}^{(i)} \simeq  \bigoplus\limits_{h = k}^n  \big(\Jex_{h,d}^{(i)}\big)^{\oplus e_{k,h}}$.
On en déduit la première égalité sur les dimensions.
La deuxième égalité sur les dimensions provient de l'inversion de la matrice binomiale 
(pour les détails, confer la preuve de~\ref{DetWhdQuotBinomialDetBkd}).
\end{proof}

\begin{rmq}
Les isomorphismes entre $\bfA[\uX]$-modules qui interviennent dans
cette proposition, qui sont des isomorphismes de $\bfA$-modules
$\bbN^n$-gradués, ne sont en aucun cas des isomorphismes de $\bfA[\uX]$-modules,
sauf cas exceptionnel. Vérifions par exemple que $\Mmac_1$ et 
$\Jex_2 \oplus \cdots \oplus \Jex_n$ ne sont pas $\bfA[\uX]$-isomorphes sauf dans le cas très
particulier de $n= 2$. Le $\bfA[\uX]$-module $\Mmac_1$ est
minimalement engendré par les $\binom{n}{2}$ générateurs 
$X_i^{d_i}\, e_j$ avec $i < j$, \og minimal \fg{} ayant le sens très précis suivant : 
le $\bfA$-module $\Mmac_1/(\bfA[\uX]_+ \Mmac_1)$ est libre de rang~$\binom{n}{2}$,
avec comme base les classes des générateurs ci-dessus. 

Quant au $\bfA[\uX]$-module $\Jex_h$, il est minimalement engendré
par les~$\binom{n}{h}$ générateurs $X^{D(I)}$ avec ${\#I=h}$. 
Avec la précision sur la minimalité, le $\bfA$-module $\Jex_h/(\bfA[\uX]_+\Jex_h)$ est libre de rang
$\binom{n}{h}$, avec comme base les classes des~$X^{D(I)}$. 
Ainsi :
$$
\dim_\bfA \dfrac{\Mmac_1}{\bfA[\uX]_+ \Mmac_1} = \binom{n}{2}
\qquad \text{et} \qquad
\dim_\bfA \frac{\Jex_2\oplus\cdots\oplus\Jex_n}{\bfA[\uX]_+(\Jex_2\oplus\cdots\oplus\Jex_n)} =
\binom{n}{2} + \binom{n}{3} + \cdots + \binom{n}{n}
$$
Par conséquent, les deux $\bfA[\uX]$-modules en cause sont
$\bfA[\uX]$-isomorphes seulement dans le cas $n=2$.
\end{rmq}

\begin{prop}\label{DecompositionKkdParJhd}
On dispose d'un isomorphisme de $\bfA$-modules : 
$$
\rmK_{k,d}
\ \simeq \ 
\bigoplus\limits_{h = k}^n  \ \Jex_{h,d}^{\oplus e'_{k,h}}
\qquad 
\text{où \  $e'_{k,h} = \binom{h-1}{h-k}$}
$$
\end{prop}

\begin{proof}
En utilisant que $\rmK_{k,d} = \Mmac_{k,d} \oplus \Smac_{k,d} 
\simeq \Mmac_{k,d} \oplus \Mmac_{k-1,d}$, 
on obtient :
$$
\rmK_{k,d} 
\quad \simeq \quad 
\bigoplus\limits_{h = k+1}^n  \ \Jex_{h,d}^{\oplus e_{k+1,h}}
\ \oplus \ 
\bigoplus\limits_{h = k}^n  \ \Jex_{h,d}^{\oplus e_{k,h}}  
\quad = \quad 
\bigoplus\limits_{h = k}^n  \ \Jex_{h,d}^{\oplus e_{k+1,h}}
\ \oplus \ 
\bigoplus\limits_{h = k}^n  \ \Jex_{h,d}^{\oplus e_{k,h}} \\
$$
L'égalité de droite provient du fait que $e_{k+1,k} = 0$.
On conclut en remarquant que $e_{k+1,h} + e_{k,h} = e'_{k,h}$.
\end{proof}

\subsection{Relations binomiales déterminantales}

Il est temps de récolter le fruit de notre travail.
Ici, on fait intervenir l'endomorphisme~$B_k = B_k(\uP)$ uniquement sur 
la \Rdecomposition{} de $\Mmac_{k-1}$.

\begin{prop} \label{BestTriangulaireEnR}
L'endomorphisme $B_k$ est triangulaire relativement à 
$\Mmac_{k-1} = \bigoplus\limits_{\#R=k-1} \Mmac^R$, le sens de la relation d'ordre
étant le suivant:
$$
B_k\big(\Mmac^R\big) 
\ \subset \ 
\bigoplus_{R' \succcurlyeq R}\,  \Mmac^{R'}
$$
En conséquence:
$$
\det B_{k,d} =  \prod_{\#R = k-1} \det B^R_{k,d}
$$
\end{prop}

\begin{proof}
Commençons par la fin i.e. par le \og En conséquence \fg{}.
Est-il utile de préciser que la stratégie triangulaire a encore frappé
et que l'égalité déterminantale découle de~\ref{DecompositionTriangulaireDef}?

\medskip

Occupons-nous maintenant de la structure triangulaire.
Soit ${X^\beta e_J \in \Mmac^R}$. 
Notons $i = \minDiv(X^\beta)$ de sorte que  $i < \min J$.
D'après~\ref{ExpressionBkXbetaeJ}, $B_k(X^\beta e_J)$ est 
une combinaison $\bfA$-linéaire de monômes extérieurs de $\Mmac_{k-1}$ du type 
$$
X^{\beta'} e_{J'} 
\ \overset{\rm def}{=} \ 
\dfrac{X^\beta}{X_i^{d_i}} \ X^\gamma \ 
e_{(i \vee J)\setminus j} \
\qquad 
\text{avec \quad $j \in i \vee J$ \quad et \quad $X^\gamma$ un certain monôme de $P_j$}
$$
Prenons $X^{\beta'} e_{J'}$ un tel monôme ; il est dans un unique $\Mmac^{R'}$ et 
l'objectif est de montrer que $R' \succcurlyeq R$.
Rappelons que 
$$
R \ = \ \Ind\big(J \,; \DivSeq(X^{\alpha})\big) 
\qquad 
R' \ = \ \Ind\big(J' \,; \DivSeq(X^{\alpha'})\big) 
\qquad 
\text{%
où $X^\alpha 
\ \overset{\rm def}{=} \ 
X^{\beta} X^{D(J)}$ 
et 
$X^{\alpha'} 
\ \overset{\rm def}{=} \ 
X^{\beta'} X^{D(J')}$ 
}
$$
Dans la définition de $X^{\alpha'}$, en remplaçant $X^{\beta'}$ par sa valeur puis $X^\beta$ par $X^\alpha/X^{D(J)}$:
$$
X^{\alpha'} \ = \ \dfrac{X^{D(J')} X^\alpha}{X^{D(J)}\,X_i^{d_i}} X^\gamma
\qquad \text{qui conduit à} \qquad
X^{\alpha'} \overset {\textstyle (\star)}{=}
\left\{
\begin {array}{cl}
\dfrac{X^\alpha}{X_i^{d_i}} X^\gamma   &\text{si $j=i$ \idest{}  $J'=J$} \\[5mm]
\dfrac{X^\alpha}{X_j^{d_j}} X^\gamma   &\text{si $j\in J$ \idest{} $J'= i\vee J \setminus j$} \\
\end {array}
\right.
$$
On veut montrer que $R \preccurlyeq R'$. Pour cela, on introduit 
la partie finissante $F = \DivSeq(X^\alpha)_{\geqslant \min J}$ de sorte que $R = \Ind(J \,; F)$ 
d'après la propriété (A) du lemme~\ref{EltVersusInd}. Idem avec des primes. 
Il ne reste plus qu'à montrer que $\Ind(J \,; F) \preccurlyeq \Ind(J' \,; F')$. On distingue deux cas 
suivant la valeur de $j \in i \vee J$.

\smallskip
\noindent
$\rhd$
Le cas $j= i$. 
Grâce à l'égalité $(\star)$, on a $\DivSeq(X^\alpha) \setminus i \subset \DivSeq(X^{\alpha'})$.
Comme $i < \min J$, on obtient $F \subset \DivSeq(X^{\alpha'})$ donc $F \subset F'$.
On donne un coup de $\Ind(J \,; \sbullet)$ et on obtient, grâce à la propriété (B) du lemme~\ref{EltVersusInd}, 
$\Ind(J \,; F) \preccurlyeq \Ind(J' \,; F')$.

\smallskip
\noindent
$\rhd$
Le cas $j \in J$.
Dans ce cas, $i < \min F$ (car $\min F = \min J$).
Le point~(C) du lemme~\ref{EltVersusInd} donne donc
$$
\Ind(J \,; F) \preccurlyeq \Ind(J' \,; i \vee F \setminus j)
$$ 
Ensuite, en reprenant l'égalité $(\star)$, on a 
$\DivSeq(X^\alpha) \setminus j \subset \DivSeq(X^{\alpha'})$.
Comme $i \in \DivSeq(X^\alpha) \setminus j$ et $\min J' = i < \min J$, on en déduit 
$i \vee F \setminus j \subset F'$.
Le point (B) du lemme~\ref{EltVersusInd} fournit alors
$$
\Ind(J' \,; i \vee F \setminus j) \preccurlyeq \Ind(J' \,; F')
$$
En combinant les deux $\preccurlyeq$-inégalités, on récupère donc $\Ind(J \,; F) \preccurlyeq \Ind(J' \,; F')$.
\end{proof}

\begin{theo}[$\det B_{k,d}$ comme produit binomial des $(\det W_{h,d})_{h\ge k}$]
\label{DetBkdProdBinomialDetWhd}
Soit $d \in \bbN$.

\begin {enumerate} [\rm i)]
\item
On dispose de l'identité algébrique :
$$
\det B_{k,d} \ = \ 
\prod_{h = k}^n \ \big(\det W_{h,d}\big)^{e_{k,h}}
\qquad 
\text{où \  $e_{k,h} = \binom{h-2}{h-k}$} 
$$
En particulier pour $k=1,2$:
$$
\det B_{1,d} = \det W_{1,d}, \qquad\qquad
\det B_{2,d} = \det W_{2,d}\,\det W_{3,d}\,\cdots \det W_{n,d}
$$
et pour $k = 3$:
$$
\det B_{3,d} = \prod_{h=3}^n \ \big(\det W_{h,d}\big)^{h-2} =
(\det W_{3,d})^1\, (\det W_{4,d})^2\, (\det W_{5,d})^3 \cdots
$$

\item
On a une égalité analogue dans $\bfA[t]$ pour les polynômes caractéristiques 
$$
\chi_{_{B_{k,d}}}(t) = 
\prod_{h = k}^n  \Big(\chi_{_{W_{h,d}}}(t)\Big)^{e_{k,h}}
$$
En considérant les degrés des polynômes caractéristiques, on obtient l'égalité dimensionnelle
$$
\dim \Mmac_{k-1,d} = \sum_{h=k}^n e_{k,h}\,\dim \Jex_{h,d}
$$
\end {enumerate}
\end{theo}

\index{relations binomiales!entre $\det B_{k,d}$ et les $(\det W_{h,d})_{h\ge k}$}%

\begin{proof} \leavevmode

i)  
D'après~\ref{BestTriangulaireEnR}, on a 
$$
\det B_{k,d} \ = \ 
\prod_{\#R = k-1} \det B_{k,d}^R 
$$
D'après~\ref{BRquasi-conjWh}, on a
$\det B_{k,d}^R = \det W_{h,d}$ où $h = 1+\max R \in \{k..n\}$.
Supposons $k \geqslant 2$.
Pour $h$ fixé, il y a $\binom{h-2}{k-2}$ parties $R \subset \{1..n\moins1\}$ de cardinal $k-1$ et 
de maximum $h-1$. Ce coefficient binomial vaut aussi $e_{k,h} = \binom{h-2}{h-k}$.
Pour $k = 1$, combien de telles parties $R$ ?
On doit avoir $R = \emptyset$ et $h = 1 + \max R$ vaut $1 + 0$ par convention.
Cela donne $1$ pour $h = 1$ et $0$ sinon, ce qui est exactement le coefficient binomial~$e_{1,h}$.

\medskip
ii) Considérer $t\,X_i^{d_i} - P_i$ au lieu de $P_i$.
\end{proof}

La formule suivante est symbolique (comme pratiquement toutes celles de cette étude) : 
l'écriture $a = \dfrac{b}{c}$ code l'égalité $a \times c = b$.
%
\begin{prop} [$\det W_{h,d}$ comme quotient alterné binomial des $(\det B_{k,d})_{k \ge h}$]
\label{DetWhdQuotBinomialDetBkd}
Pour tout entier $d$, on a l'identité algébrique :
$$
\det W_{h,d} \ = \ 
\prod_{k \geqslant h} \ \big(\det B_{k,d}\big)^{(-1)^{k-h} e_{h,k}}
$$
En particulier, pour $h = 1,2$: 
$$
\det W_{1,d} \ =\  \det B_{1,d},
\qquad \quad 
\det W_{2,d}
\ = \ 
\dfrac{\det B_{2,d} \,\det B_{4, d} \, \cdots}
{\det B_{3, d} \, \det B_{5, d} \, \cdots}
$$
Et pour $h=3$:
$$
\det W_{3,d}
\ = \
\prod_{k \geqslant 3} \ \big(\det B_{k,d}\big)^{(-1)^{k-3} (k-2)}
\ = \
\dfrac{(\det B_{3,d})^1 \, (\det B_{5, d})^3 \, \cdots}
{(\det B_{4, d})^2 \, (\det B_{6, d})^4 \, \cdots}
$$
\end{prop}

\index{relations binomiales!entre $\det W_{h,d}$ et les $(\det B_{k,d})_{k\ge h}$}%

\begin{proof}
Nous allons utiliser la matrice dite binomiale qui figure dans
la proposition~\ref{LaMatriceBinomiale} suivante.  C'est la
matrice carrée indexée par $\{-1..n-2\}$ dont le coefficient
$(i,j)$ est $\binom{j}{j-i}$ et dont l'inverse a pour terme
$(-1)^{j-i}\binom{j}{j-i}$.

Ici, la matrice carrée qui intervient est indexée par $\{1..n\}$ et son
coefficient $(k,h)$ est $e_{k,h} = \binom{h-2}{h-k}$.  Il suffit
d'effectuer la translation $j=h-2$ et $i=k-2$.
Le terme $(h,k)$ de l'inverse est donc $(-1)^{k-h} e_{h,k}$.
\end{proof}

\subsubsection{La matrice binomiale}

Rappelons (cf. le chapitre~\ref{ObjetsSuiteP}, avant la proposition~\ref{dminKk})
que nous avons adopté  la définition suivante pour le coefficient binomial, cf.~\cite[p. 154]{GKP}.
Pour $u,x \in \bbZ$:
$$
\binom{x}{u} \ = \ 
\left\{
\begin{array}{ll}
\dfrac{x(x-1)\cdots(x-u+1)}{u(u-1) \cdots1} & \text{si $u \geqslant 0$} \\
0 & \text{si $u<0$} \\
\end{array}
\right.
$$
En particulier, pour tout $x \in \bbZ$, on a  $\binom{x}{0} = 1$, $\binom{x}{1} = x$ et $\binom{x}{x} = 0$
si $x < 0$, $1$ si $x \ge 0$.
Pour $u \in \bbN$ et $x$ une indéterminée sur $\bbQ$, la formule ci-dessus
définit le polynôme $\binom{x}{u} \in \bbQ[x]$, qui est un polynôme en $x$ de degré~$u$.

\begin{prop}\label{LaMatriceBinomiale}
Soient $a, b \in \mathbb Z$ tels que $a \leqslant b$ et $\calE$ l'intervalle $\{a,\dots,b\}$.
La matrice $\calE \times \calE$ dont le coefficient $(i,j)$ vaut $\binom{j}{j-i}$ est inversible, 
et le coefficient $(i,j)$ de son inverse vaut $(-1)^{j-i} \binom{j}{j-i}$.
\end{prop}

Avant de proposer la preuve dans le cas où $a \geqslant -1$, montrons plusieurs exemples avec $b=6$.
Ci-dessous, respectivement $a=-1$, $a=0$ et $a=2$ :
$$
\left[
\begin{array}{*{8}{c}}
1& .& .& .& .& .& .& . \\ 
.& 1& 1& 1& 1& 1& 1& 1 \\ 
.& .& 1& 2& 3& 4& 5& 6 \\ 
.& .& .& 1& 3& 6& 10& 15 \\ 
.& .& .& .& 1& 4& 10& 20 \\ 
.& .& .& .& .& 1& 5& 15 \\ 
.& .& .& .& .& .& 1& 6 \\ 
.& .& .& .& .& .& .& 1 \\ 
\end{array}
\right]
\qquad
\left[
\begin{array}{*{7}{c}}
1& 1& 1& 1& 1& 1& 1 \\ 
.& 1& 2& 3& 4& 5& 6 \\ 
.& .& 1& 3& 6& 10& 15 \\ 
.& .& .& 1& 4& 10& 20 \\ 
.& .& .& .& 1& 5& 15 \\ 
.& .& .& .& .& 1& 6 \\ 
.& .& .& .& .& .& 1 \\ 
\end{array}
\right]
\qquad 
\left[
\begin{array}{*{5}{c}}
1& 3& 6& 10& 15 \\ 
.& 1& 4& 10& 20 \\ 
.& .& 1& 5& 15 \\ 
.& .& .& 1& 6 \\ 
.& .& .& .& 1 \\ 
\end{array}
\right]
$$
Les matrices de droite sont des sous-matrices de la matrice de gauche.
L'inverse pour $a=-1$ est donné ci-dessous avec la convention $\overline {m} = -m$.
On en déduit l'inverse pour $a=0$ et pour $a=2$ en prenant la sous-matrice ad-hoc.
$$
\def\ov{\overline}
\left[
\begin{array}{*{8}{c}}
1& .& .& .& .& .& .& . \\ 
.& 1& \ov1& 1& \ov1& 1& \ov1& 1 \\ 
.& .& 1& \ov2& 3& \ov4& 5& \ov6 \\ 
.& .& .& 1& \ov3& 6& \ov{10}& 15 \\ 
.& .& .& .& 1& \ov4& 10& \ov{20} \\ 
.& .& .& .& .& 1& \ov5& 15 \\ 
.& .& .& .& .& .& 1& \ov6 \\ 
.& .& .& .& .& .& .& 1 \\ 
\end{array}
\right]
$$

\begin{proof}[Preuve pour $a \geqslant -1$]
Le cas où $a = 0$ correspond à la matrice de 
l'endomorphisme $P(X) \mapsto P(X+1)$ dans la base canonique $(1,X, \dots, X^{b})$:
$$
(X+1)^j = \sum_{i=0}^j \binom {j}{i} X^i = \sum_{i=0}^b \binom {j}{j-i} X^i,
\qquad  0 \le j \le b
$$
Cet endomorphisme est bijectif d'inverse $P(X) \mapsto P(X-1)$. 
La formule $(X-1)^j = \sum_{i=0}^j \binom{j}{j-i} (-1)^{j-i} X^i$ permet 
d'obtenir le coefficient $(i,j)$ de l'inverse.

La matrice pour $a = -1$ est obtenue en rajoutant à la matrice précédente
un bloc diagonal de taille $1 \times 1$ constitué de $\binom{-1}{0} = 1$, donc 
cela ne modifie pas l'inversibilité et l'inverse est facile à calculer à partir 
de l'inverse de la matrice précédente.
\end{proof}

Nous avons limité la preuve au cas $a \ge -1$ car nous n'avons besoin que de ce cas. Nous
laissons le soin au lecteur de traiter le cas où $a$ est \og très négatif\fg{} en lui montrant
la matrice pour $a=-5$ ($b$ étant toujours égal à $6$):
$$
\def\ov{\overline}
\left[
\begin{array}{*{12}{c}}
1& \ov4& 6& \ov4& 1& .& .& .& .& .& .& . \\ 
.& 1& \ov3& 3& \ov1& .& .& .& .& .& .& . \\ 
.& .& 1& \ov2& 1& .& .& .& .& .& .& . \\ 
.& .& .& 1& \ov1& .& .& .& .& .& .& . \\ 
.& .& .& .& 1& .& .& .& .& .& .& . \\ 
.& .& .& .& .& 1& 1& 1& 1& 1& 1& 1 \\ 
.& .& .& .& .& .& 1& 2& 3& 4& 5& 6 \\ 
.& .& .& .& .& .& .& 1& 3& 6& 10& 15 \\ 
.& .& .& .& .& .& .& .& 1& 4& 10& 20 \\ 
.& .& .& .& .& .& .& .& .& 1& 5& 15 \\ 
.& .& .& .& .& .& .& .& .& .& 1& 6 \\ 
.& .& .& .& .& .& .& .& .& .& .& 1 \\ 
\end{array}
\right]
$$



Le lemme ci-dessous va être utilisé dans le théorème suivant pour
fournir une expression binomiale multiplicative de $\Delta_{k,d}$ (où
$k \ge 1$) en fonction des $(\det W_{h,d})_{h \ge k}$.

\begin{lem} \label{JolieEgaliteCoeffsBinomiaux} 
Pour $p \in \bbN$, on a l'identité polynomiale en $x$
$$
\binom {x}{p} \ =\ (-1)^p \sum_{i=0}^p (-1)^i\binom {x+1}{i}
$$
\end{lem}

\begin{proof}
En utilisant la relation de Pascal, on obtient une somme télescopique :
$$
\sum_{i=0}^p (-1)^i \binom{x+1}i
\ = \ 
\sum_{i=0}^p (-1)^i\binom{x}i
- 
\sum_{i=0}^p (-1)^{i-1} \binom{x+1}{i-1}
\ = \ 
(-1)^p \binom{x}p
$$
\end{proof}

Nous pouvons enfin énoncer un des résultats les plus importants de ce
chapitre.  Il convient d'y corriger un petit mensonge dans le titre
car l'expression correspondant à $k=1$ est $\Delta_{1,d} = \det
W^\sigma_{1,d} / \det W^\sigma_{2,d}$, que l'on ne peut pas vraiment
qualifier de produit comme dans le cas $k \ge 2$.

\index{relations binomiales!z@entre $\Delta_{k,d}$ et les $(\det W_{h,d})_{h\ge k}$}%

\begin{theo} [$\Delta_{k,d}$ comme produit binomial des $(\det W_{h,d})_{h \ge k}$]
\label{DeltakdProdBinomialDetWhd}
Pour $\sigma \in \mathfrak S_n$ et $k \geqslant 1$, on dispose de la formule close suivante :
$$
\Delta_{k,d}^\sigma \ = \ 
\prod_{h \geqslant k} \ \big(\!\det W_{h,d}^\sigma\big)^{f_{k,h}}
\qquad 
\text{où \  $f_{k,h} = \binom{h-3}{h-k}$}
$$
En particulier pour $k=1$, on a:
$$
\Delta_{1,d} = \det W^\sigma_{1,d} (\det W^\sigma_{2,d})^{-1}
\qquad \text{ou encore, sans division} \qquad
\det W^\sigma_{1,d} = \Delta_{1,d}\,\det W^\sigma_{2,d}
$$
Un autre cas particulier est extrêment important, c'est celui de $k =
2$ : $\boxed{\Delta_{2,d}^\sigma = \det W_{2,d}^\sigma}$.

Et pour $k=3$, nous avons:
$$
\Delta_{3,d}^\sigma =
\det W_{3,d}^\sigma\, \det W_{4,d}^\sigma \cdots \det W_{n,d}^\sigma 
$$
Note: la première formule, quitte à assurer un petit décalage, peut
être écrite à l'aide du coefficient binomial habituel $e_{k,h}$ de la
manière suivante, avec cette fois $k \geqslant 0$:
$$
\Delta_{k+1,d}^\sigma \ = \ 
\prod_{h \geqslant k} \ \big(\!\det W_{h+1,d}^\sigma\big)^{e_{k,h}}
\qquad 
\text{où \  $e_{k,h} = \binom{h-2}{h-k}$}
$$
\end{theo}

\label {NOTA15-fkh}%
%
%

\begin{proof}
Presque par définition (cf chapitre~\ref{ComplexeDecompose},
en particulier la proposition~\ref{SuiteDeltak} et la section~\ref{SectionDeltakdMacaulay})
le scalaire $\Delta_{k,d}$ est un quotient alterné de $\det B_{k',d}$ avec $k' \geqslant k$
$$
\Delta_{k,d}
\ = \ 
\dfrac{\det B_{k,d}\,\det B_{k+2,d}\, \cdots}
{\det B_{k+1,d}\, \det B_{k+3,d}\, \cdots}
$$
Or, d'après~\ref{DetBkdProdBinomialDetWhd}, on a 
$\det B_{k,d} = \prod\limits_{h \geqslant k} (\det W_{h,d})^{e_{k,h}}$ avec 
$e_{k,h} = \binom{h-2}{h-k}$.
On obtient donc:
$$
\Delta_{k,d} = \prod\limits_{h \geqslant k} (\det W_{h,d})^{f_{k,h}}
\qquad \text{avec} \qquad
f_{k,h} \ =\  (e_{k,h} + e_{k+2,h} + \cdots ) \ - \ 
(e_{k+1,h} + e_{k+3,h} + \cdots )
$$
Réécrivons $f_{k,h}$:
$$
f_{k,h} \ = \ 
\sum_{j = k}^h (-1)^{j-k} \, \binom{h-2}{h-j}
\ = \ 
(-1)^{h-k} 
\sum_{i = 0}^{h-k} (-1)^{i} \, \binom{h-2}{i}
$$
Or la somme de droite vaut $\binom{h-3}{h-k}$, 
d'après le lemme~\ref{JolieEgaliteCoeffsBinomiaux} 
dans lequel on a effectué la substitution $p = h-k$ et $x = h-3$.
Bilan: $f_{k,h} = \binom{h-3}{h-k}$,  ce qui termine la preuve
de la première égalité.

\medskip

Pour $k = 1$, voici les $(f_{k,h})_{h\ge k}$ 
$$  
f_{1,1} = 1, \qquad  f_{1,2} = -1, \qquad f_{1,h} = 0 \quad \forall h \ge 3
$$
Ceci conduit à l'égalité annoncée à condition d'utiliser l'égalité
$\Delta^\sigma_{1,d} = \Delta_{1,d}$, valide pour tout $d \in \bbN$ et
tout $\sigma \in \fS_n$ (cf.~\ref{EgaliteDelta1dDelta1dsigma}).

\smallskip

Le cas particulier $k=2$ résulte du fait que $f_{2,2} = 1$ et $f_{2,h} = 0$ pour $h \ge 3$.
Quant à $k=3$, on a $f_{3,h} = 1$ pour $h \ge 3$.

\medskip
La dernière égalité, assez anecdotique, s'obtient en posant dans la première $k=k'+1$, $h=h'+1$
et en utilisant $f_{k'+1,h'+1} = e_{k',h'}$.
\end{proof}

\medskip

On remarque que cette formule close ne présente pas de dénominateur
pour $k \geqslant 2$ et permet donc de
déterminer~$\Delta_{k,d}^\sigma$ pour toute suite $\uP$, sans problème
de régularité et/ou division par zéro.

\smallskip
En revanche, pour $k=1$, intervient une division; nous pouvons
cependant formuler le résultat sans division comme nous l'avons fait
dans le théorème.  Vu l'importance du cas $k=1$, nous lui
consacrons le corollaire suivant ne présentant rien de nouveau vu
qu'il est contenu dans le théorème précédent.  Nous insistons encore
sur le fait d'exprimer symboliquement $\Delta_{1,d}$ à l'aide d'une
famille de quotients indexée par $\fS_n$.  Enfin, pour ceux qui
n'ont pas apprécié le coup du coefficient binomial $f_{k,h}$ poussé dans
ses retranchements pour $k=1$, nous fournissons un autre argument
d'une ligne.

\medskip

\begin {coro}
\label{DetW2diviseDetW1}
Pour tout $d$:
$$
\det W_{1,d} = \Delta_{1,d}\, \det W_{2,d}
$$
Plus généralement, pour tout $\sigma \in \fS_n$:
$$
\det W^\sigma_{1,d} = \Delta_{1,d}\, \det W^\sigma_{2,d}
$$
Ce que l'on écrit de manière symbolique:
$$
\Delta_{1,d} = \frac {\det W_{1,d}}{\det W_{2,d}} = \frac{\det W^\sigma_{1,d}} {\det W^\sigma_{2,d}}
$$  
En particulier, $\det W^\sigma_{2,d} \mid \det W^\sigma_{1,d}$.
\end {coro}

\begin {proof}
Résulte de $\det W_{1,d} = \det B_{1,d} = \Delta_{1,d} \Delta_{2,d}$ et de l'égalité
$\Delta_{2,d} = \det W_{2,d}$.
Idem pour la $\sigma$-version.
\end {proof}

\medskip

On peut désormais remplacer $\Delta_{2,d}$ par $\det W_{2,d}$ dans les
énoncés antérieurs, par exemple dans la définition/proposition
\ref{OmegaDetNumerator} relative à $\Det_d$.  L'énoncé ci-dessous
permet de mettre en veilleuse la section
\ref{sousSectionMacaulayRecurrence}. Encore une fois (ce n'est
probablement pas la dernière), nous utilisons, à propos des quotients,
le petit couplet \og écriture symbolique\fg.

\begin {theo} [\og Formule de Macaulay généralisée\fg]
\label {MegaMacaulayTheorem}
\leavevmode

Soit $\sigma \in \fS_n$. Pour $d \in \bbN$:
$$
\Det_d(f) = \frac{\det\big(\bsOmega_d(f)\big)}{\det W_{2,d}} =
\frac{\det\big(\bsOmega^\sigma_d(f)\big)}{\det W^\sigma_{2,d}}
\qquad \forall\ f \in \AXdSod
$$
manière symbolique d'écrire les égalités sans dénominateurs:
$$
\det\big(\bsOmega_d(\sbullet)\big) = \det W_{2,d} \Det_d(\sbullet), \qquad
\det\big(\bsOmega^\sigma_d(\sbullet)\big) = \det W^\sigma_{2,d} \Det_d(\sbullet)
$$
Bien entendu, pour $d=\delta$, les paroles relatives à $\omegares :
\bfA[\uX]_\delta \to \bfA$ sont les suivantes:
$$
\omegares = \dfrac{\omega}{\det W_{2,\delta}} =
\dfrac{\omega^\sigma}{\det W^\sigma_{2,\delta}}
$$
Et en degré $d \ge \delta+1$, on commence à connaître la musique concernant l'expression $\calR_d$
de $\Res(\uP)$:
$$
\calR_d = \dfrac{\det W_{1,d}}{\det W_{2,d}}  = \dfrac{\det W^\sigma_{1,d}}{\det W^\sigma_{2,d}}
$$
\end {theo}  

\index{formules de Macaulay}%

\cleardoublepage
\section {Les déterminants de Sylvester en degré $\delta - \min(D) < d \leqslant \delta$}
\label{ChapSylvesterHybride}

Dans ce chapitre, sauf mention explicite du contraire, $d$ est un
entier fixé dans la fourchette:
$$
\boxed {\delta - \min(D) < d \leqslant \delta}
$$
Nous notons $d' = \delta-d$ l'entier complémentaire de
$d$ à $\delta$ de sorte que la contrainte encadrée sur $d$ équivaut à
$0 \leqslant d' < \min(D)$.

\medskip

Nous allons en particulier exhiber une famille $\bsnabla_{d'}(\uP) =
\big(\nabla_\beta(\uP)\big)_{|\beta| = d'}$ de déterminants homogènes de degré~$d$,
définis modulo $\langle\uP\rangle$, qui sont des formes d'inertie de
$\langle \uP \rangle$.  Cette terminologie ``forme d'inertie'', qui
désigne un polynôme homogène de $\uPsat$, est celle utilisée par
Jouanolou, en particulier dans le titre de son article \cite{J7}.
Nous disposons ainsi d'une famille bien définie d'éléments de
$$
\vH^0_\uX(\bfB)_d = \uPsat_d /\langle\uP\rangle_d \quad\subset\quad
\bfB_d = \bfA[\uX]_d/\langle\uP\rangle_d
$$
Lorsque $\uP$ est régulière, nous montrerons que cette famille est une
$\bfA$-base de $\vH^0_\uX(\bfB)_d$.  Pour $d=\delta$, on verra que
l'on retrouve les ingrédients du monde en degré $\delta$ puisque 
$\nabla_{\mathbf 0}$ coïncide exactement avec notre bon vieux
déterminant bezoutien $\nabla \in \bfB_\delta$.

\medskip

L'inégalité sur $d'$ a comme conséquence que $\Jex_{1,d'} = 0$, d'où l'égalité
$\Smac_{0,d'} = \bfA[\uX]_{d'}$. L'utilisation de l'isomorphisme mouton-swap
$$
\begin{array}{rcl}
\Smac_{0,d} & \overset{\simeq}{\longrightarrow} & \Smac_{0,d'} \\ [1.5mm]
X^\alpha & \longmapsto & X^\beta = X^{\emouton - \alpha}
\end{array}
$$
va nous permettre de passer d'une famille indexée par 
la base monomiale $\calS_{0,d}$ de $\Smac_{0,d}$ à 
une famille indexée par la base monomiale $\calS_{0,d'}$
de $\bfA[\uX]_{d'}$ et réciproquement.

\medskip

Par ailleurs, nous rappelons que, pour tout $d$, nous avons défini
en~\ref{CanonicalCayleyDetSection}, la forme $r_{0,d}$-linéaire
canonique du complexe $\rmK_{\sbullet,d}(\uP)$, que nous avons notée
$\Det_{d,\uP} : \BdSod \to \bfA$ et qui va intervenir de manière
primordiale dans ce chapitre.
Sous la contrainte imposée sur $d$ (l'inégalité encadrée au début), nous montrerons que:
$$
\boxed{\Res(\uP) = \swDet_{d,\uP}(\bsnabla_{d'})}
$$
A noter la présence de sw, acronyme de swap qui sera précisé en~\ref{swNotation}
et dans la définition \ref{swbsOmega}. Il s'agit
de signifier que nous évaluons $\Det_d$ en la fonction:
$$
\begin{array}{rcl}
\calS_{0,d} & \longrightarrow & \bfB_d \\ [1.5mm]
X^\alpha & \longmapsto & \nabla_{\emouton-\alpha}(\uP)
\end{array} 
$$
La détermination effective de $\Det_{d,\uP}$, valide pour n'importe
quel $d$, a été étudiée dans la section
\ref{CanonicalCayleyDetSection}.  Basée sur la décomposition de
Macaulay de $\rmK_{\sbullet,d}(\uP)$, elle a fait intervenir un quotient
avec un dénominateur $\Delta_2(\uP)$ qui s'est révélé être
égal à $\det W_{2,d}(\uP)$ (cf. le
théorème~\ref{DeltakdProdBinomialDetWhd}).
Idem pour la version tordue par $\sigma \in \fS_n$, 
et avec bien entendu un numérateur déterminantal explicite,
ce qui permet d'écrire ici:
$$
\forall\,\sigma \in \fS_n, \qquad 
\Res(\uP) \ =\ \frac{\det \swbsOmega_{d,\uP}(\bsnabla_{d'})}{\det W_{2,d}(\uP)} \ =\ 
\dfrac{\det\bsOmega^{\sigma\,\sw}_{d,\uP}(\bsnabla_{d'})}{\det W^\sigma_{2,d}(\uP)}
$$
Notre étude s'inspire de la section~3.10
\og Formes de Sylvester et applications\fg{} de~\cite{J7}, 
dont la lecture a été très bénéfique.  Il faut cependant noter que
notre étude va plus loin que celle de Jean-Pierre Jouanolou qui n'a
pas déterminé le multiplicateur $c_\nu(\uP)$ défini dans 3.10.14.c.
Ici nous voyons que ce multiplicateur est $\det W_{2,d}(\uP)$.
Attention, l'entier $\nu$ de Jouanolou 
correspond à notre $d'$, et non pas à notre $d$.

D'autre part, nous estimons que certaines de nos preuves sont plus
simples que celles figurant chez cet auteur.  Les déterminants
$\nabla_\beta$ et $\det\big(\swbsOmega_d(\bsnabla_{d'})\big)$ qui
interviennent dans ce chapitre sont attribués à Sylvester
(cf. l'introduction de~\cite{J7}).

\subsection{Matrices $\beta${}-bezoutiennes pour $|\beta| < \min(D)$}

Notons $\uX^{\beta+\Un}$ la suite $(X_1^{\beta_1+1}, \dots, X_n^{\beta_n+1})$
où $\beta \in \bbN^n$. On s'intéresse ici à l'existence d'une
matrice $\dsV \in \bbM_n(\bfA[\uX])$ telle que $[\uP] = [\uX^{\beta+\Un}].\dsV$,
ce qui équivaut à l'inclusion pour chaque $i$:
$$
\bfA[\uX]_{d_i} \subset \langle\uX^{\beta+1}\rangle
$$
Or pour un entier $c$ donné, puisque $\bfA[\uX]_{\geqslant c}$ est l'idéal
engendré par le $\bfA$-module $\bfA[\uX]_c$, l'inclusion
$\bfA[\uX]_c \subset \langle\uX^{\beta+\Un}\rangle$ est équivalente à
$\bfA[\uX]_{\geqslant c} \subset \langle\uX^{\beta+\Un}\rangle$.  Supposons
qu'elle ait lieu: comme $X^\beta \notin \langle
\uX^{\beta+\Un} \rangle$, on tire
$X^\beta \notin \bfA[\uX]_{\geqslant c}$ i.e. $|\beta| < c$.
La proposition suivante étudie et précise la réciproque.

\label{NOTA16-XbetaPlusUn}%
%
%

\begin{prop}\label{MatriceBetaBezoutienne}
\leavevmode

\medskip
Soit $\beta \in \bbN^n$ vérifiant $d' := |\beta| < \min(D)$ et $d = \delta-d'$.

\begin{enumerate}[\rm i)]
\item 
On a l'inclusion pour tout $i$:
$$
\bfA[\uX]_{d_i} \ \subset\  \langle \uX^{\beta+\Un} \rangle
$$

\item 
Il existe une matrice homogène $\dsV_\beta \in
\bbM_n(\bfA[\uX])$ transformant la {\rm ligne} $\uX^{\beta+\Un}$
en la {\rm ligne} $\uP$ au sens suivant:
$$
\begin{bmatrix}
P_1 & \cdots &  P_n
\end{bmatrix} 
\ = \ 
\begin{bmatrix}
X_1^{\beta_1+1} & \cdots &  X_n^{\beta_n+1}
\end{bmatrix}
\dsV_{\beta}
$$
\og homogène \fg{} signifiant que $(\dsV_\beta)_{ij}$ est un polynôme homogène de degré $d_j - (\beta_i+1)$.

Une telle matrice sera appelée {\rm matrice $\beta$-bezoutienne} de $\uP$.
Elle peut toujours être prise à coefficients dans le
sous-anneau engendré (sur $\bbZ$) par les coefficients des $P_i$.

\item 
On note $\nabla_\beta = \det\dsV_\beta$. 
C'est un polynôme homogène de $\uPsat_d$ dont la classe modulo $\langle\uP\rangle_d$
est indépendante du choix de $\dsV_\beta$.

\item
Pour le jeu étalon $\uX^D$, on peut prendre pour $\dsV_\beta$ la
matrice diagonale $\diag(\uX^{\emouton - \beta})$. Pour ce choix, on a
$\nabla_\beta(\uX^D) = X^{\emouton-\beta}$.

\item
Pour $\uP$ générique, il existe une matrice $\beta$-bezoutienne
$\dsV_\beta$ pour $\uP$ qui se spécialise en $\diag(\uX^{\emouton -
\beta})$ lorsque l'on spécialise~$\uP$ en le jeu étalon~$\uX^D$,
de sorte que $\det(\dsV_\beta)$ se spécialise en $X^{\emouton-\beta}$.
\end{enumerate}
\end{prop}

\label{NOTA16-Vbeta}%
\label{NOTA16-nablabeta}%
%
%

\begin{proof} \leavevmode

Montrons l'implication générale, valable pour tout $\beta \in \bbN^n$ 
et tout $c \in \bbN$ 
$$
|\beta| < c 
\qquad \Longrightarrow \qquad 
\bfA[\uX]_{c} \ \subset \ \langle \uX^{\beta+\Un} \rangle
$$
Soit $X^\gamma \in \bfA[\uX]_c$.
Comme $|\beta| < |\gamma| = c$, il existe 
un indice $j$ tel que $\beta_j < \gamma_j$ donc 
$X^\gamma \in \langle \uX^{\beta+\Un}\rangle$.
Pour montrer le point i) proprement dit, 
il suffit d'appliquer ce qui précède à $c = d_i$ car $|\beta| < \min(D)$.

\medskip
iii)
Puisque $(X_1^{\beta_1+1}, \dots, X_n^{\beta_n+1})$ est une suite régulière,
elle est 1-sécante, donc pour deux choix $\dsV_\beta$ et $\dsV'_\beta$, on
a $\det \dsV_\beta - \det \dsV'_\beta \in \langle \uP\rangle$
d'après le résultat \ref{IndependanceNabla}.

\medskip
v) Posons $\uP' = \uP - (p_1X_1^{d_1}, \dots, p_nX_n^{d_n})$ où $p_i =
\coeff_{X_i^{d_i}}(P_i)$ et notons $\dsV'_\beta$ une matrice
$\beta$-bezoutienne pour~$\uP'$ à coefficients dans l'anneau
engendré par les coefficients de $\uP'$ (en particulier de
diagonale nulle). La matrice suivante convient:
$$
\dsV_\beta = \dsV'_\beta + \diag(p_1X_1^{d_1-(\beta_1+1)}, \cdots, p_nX_n^{d_n-(\beta_n+1)}) 
$$
Les autres points sont laissés au lecteur.
\end{proof}

\begin {rmq} [Le choix douloureux de la terminologie et des notations] \leavevmode
\label{swNotation}

\medskip
$\rhd$
La terminologie ``matrice $\beta$-bezoutienne'' n'est certainement pas
habituelle. Nous l'avons utilisée ici de manière à pouvoir nous
exprimer plus facilement. Lorsque $\beta = 0$, on retrouve la notion
de matrice bezoutienne employée depuis le chapitre~\ref{ObjetsSuiteP}
(cf. la définition~\ref{DefNabla}), terminologie qui n'est
probablement pas classique non plus.  Mais ces matrices sont tout
de même en rapport avec les matrices bezoutiennes \og classiques\fg,
celles en deux jeux $\uX, \uY$ de variables, cf. le chapitre
ultérieur~\ref{SectionBezoutMorley}.

\medskip
$\rhd$
En ce qui concerne $\nabla_\beta := \det(\dsV_\beta)$, la littérature utilise
parfois le vocabulaire ``forme de Sylvester'', forme ayant le sens de
polynôme homogène. Ici, nous avons choisi de parler de déterminant de
Sylvester, le mot \textit{forme} étant déjà utilisé dans d'autres sens (forme linéaire,
forme $r$-linéaire alternée).

\medskip
$\rhd$
Indexation par $d$ versus indexation par $d'$.
En utilisant l'isomorphisme mouton-swap, nous allons pouvoir évaluer la
forme $r_{0,d}$-linéaire $\Det_d$ sur $\bfB_d$ définie en
\ref{CanonicalCayleyDetSection} en des familles indexées par les
monômes de degré $d'$, par exemple en la famille $\bsnabla_{d'} =
(\nabla_\beta)_{|\beta|= d'}$ des déterminants de Sylvester de $\uP$.
Ce n'est pas que nous aimons le formalisme excessif, mais nous avons
vraiment eu besoin de notations pour passer librement d'un objet à l'autre:
$$
\Smac_{0,d'} \simeq \Smac_{0,d}, \qquad 
\BdSodprime \simeq \BdSod, \qquad
\AXdSodprime \simeq \AXdSod
$$
Nous avons ainsi choisi de faire de $\Det_d$ une forme $\swDet_d$
qui est une copie de $\Det_d$ au sens suivant :
$$
\swDet_d : \BdSodprime \to \bfA \qquad
g \mapsto \Det_d(g^\sw) \quad \text{où} \quad
g^\sw : X^\alpha \mapsto g(X^{\emouton-\alpha})
$$
permettant d'écrire par exemple:
$$
\Res(\uP)  = \swDet_{d,\uP}\big((\nabla_\beta)_{|\beta|=d'}\big)
$$

D'autres objets subiront les foudres de mouton-swap, par exemple les
constructeurs matriciels strictement carrés de la section~\ref{CanonicalCayleyDetSection}
$$
g \mapsto \swbsOmega_d(g), \qquad\qquad
g \mapsto \bsOmega^{\sigma\,\sw}_d(g)
$$
cf. la définition suivante.

\medskip
$\rhd$
Mais ce n'est pas tout! La forme $\Det_d$ (ou sa copie $\swDet_d$) est une forme
$r_{0,d}$-linéaire alternée dont l'argument n'est pas un $r_{0,d}$-uplet d'éléments
de $\bfB_d$ mais une famille d'éléments de $\bfB_d$ indexée par $\calS_{0,d}$ considérée
parfois comme une  fonction sur $\calS_{0,d}$ à valeurs dans $\bfB_d$.
Il faudra donc réaliser un petit effort pour pouvoir singer par exemple 
le remplacement, dans un $r_{0,d}$-uplet, d'une composante par un élément de $\bfB_d$.

\end {rmq}

\label {NOTA16-sw-fonction}%
\label {NOTA16-swDetd}%
%
%

\begin{defn}
\label{swbsOmega}
  
Ici $d$ est quelconque et $d' = \delta-d$. 
Etant donné $g \in \AXdSodprime$, on définit $g^\sw \in \AXdSod$ par:
$$
g^\sw(X^\alpha) = g(X^{\emouton-\alpha})
$$
Cette définition est valide car $X^{\emouton-\alpha}$ appartient à 
$\Smac_{0,d'}$ lorsque $X^\alpha \in \Smac_{0,d}$ et vice versa.
On réalise ainsi un isomorphisme $\AXdSodprime \simeq \AXdSod$ qui ne dépend que
du format $D$ et permet, par composition, de transformer toute application définie sur
$\AXdSod$ en une application définie sur $\AXdSodprime$.

\medskip
Considérons par exemple, pour un système $\uP$ de format $D$, la construction
$\bsOmega_d = \bsOmega_{d,\uP}$ définie en~\ref{OmegaDetNumerator}.
On obtient ainsi, pour $g \in \AXdSodprime$, un endomorphisme
$\swbsOmega_d(g) = \swbsOmega_{d,\uP}(g)$ de~$\bfA[\uX]_d$
$$
\vcenter{
\xymatrix @C = 1.5cm {
\AXdSodprime\ar[r]^-{g \mapsto g^\sw}_{\simeq} \ar[rd]_{\swbsOmega_d}
              &\AXdSod \ar[d]^{\bsOmega_d}  \\
              &\End_\bfA(\bfA[\uX]_d\big) \\
}}
\qquad\qquad
\swbsOmega_d(g) = \bsOmega_d(g^\sw)
$$
Pour $\sigma\in\fS_n$, définition analogue utilisant
$\bsOmega_d^\sigma$ au lieu de $\bsOmega_d$, produisant une
construction notée~$\bsOmega_d^{\sigma\,\sw}$.


\end{defn}

\noindent
Commentons cette définition en prenant $X^\alpha \in \bfA[\uX]_d$ et en décrivant
son image par $\swbsOmega_d(g)$.

$\rhd$
Si $X^\alpha \in \Smac_{0,d}$, alors l'image est $g^\sw(X^\alpha)$.

\medskip

$\rhd$
Si $X^\alpha \in \Jex_{1,d}$, en notant $i = \minDiv(X^\alpha)$, son image est exactement
celle par l'application de Sylvester $\Syl_d$ du monôme extérieur
$\dfrac{X^\alpha}{X_i^{d_i}}e_i$ de~$\Smac_{1,d}$, ou encore l'image
de $X^\alpha$ par $\Syl_d \circ \varphi$ où $\varphi : \Jex_{1,d}
\rightarrow \Smac_{1,d}$ est l'isomorphisme monomial défini par
$X^\alpha \mapsto \dfrac{X^\alpha}{X_i^{d_i}}e_i$.

A gauche, un schéma donnant un aperçu de la définition de
l'endomorphisme $\swbsOmega_d(g)$.  \og Aperçu \fg{}, car sur ce
dessin, la base monomiale de l'espace d'arrivée est rangée n'importe
comment.  Pour obtenir, la vraie définition de l'endomorphisme, il
faut penser au dessin de droite, pour lequel la base monomiale de
l'espace d'arrivée est scindée en deux : celle de $\Jex_{1,d}$, puis
celle de $\Smac_{0,d}$.

{
\newcommand \absM {3.7}
\newcommand \ordM {2.4}
\newcommand \taille {8}
\newcommand \NorthUn {(0,\taille) -- (\absM, \taille)} 
\newcommand \NorthDeux {(\absM,\taille)--(\taille,\taille)}  
\newcommand \EastUn {(\taille,\taille) -- (\taille, \ordM)} 
\newcommand \East {(\taille,\taille) -- (\taille, 0)} 
\newcommand \EastDeux {(\taille, \ordM) -- (\taille,0)}
\newcommand \RestSyl[1]  {${\mathrm{restr.} \atop \mathrm{de}\,\Syl_{#1}}$} 

\centerline{
\begin{tikzpicture}[scale = 0.5]
\draw[fill=gray!20]  (0,\taille) rectangle (\absM, 0) ;
\path (0,\taille) -- (\absM, 0) node[midway] {\RestSyl{d}} ;
\draw[fill=gray!60]  (\taille,\taille) rectangle (\absM, 0) ;
\path (\taille,\taille) -- (\absM, 0) node[midway] {$g$} ; 
\path \NorthUn node[midway,above] {$\Smac_{1,d}$} ;
\path \NorthDeux node[midway,above] {$\Smac_{0,d'}$} ;
\path \East node[midway,right] {$\bfA[\uX]_d$} ;
\end{tikzpicture}
\qquad \qquad \qquad 
\begin{tikzpicture}[scale = 0.5]
\draw[fill=gray!20]  (0,\taille) rectangle (\absM, 0) ;
\path (0,\taille) -- (\absM, 0) node[midway] {$\Syl_d \circ \varphi$} ;
\draw[fill=gray!60]  (\taille,\taille) rectangle (\absM, 0) ;
\path (\taille,\taille) -- (\absM, 0) node[midway] {$g^\sw$} ; 
\path \EastUn node[midway,right] {$\Jex_{1,d}$} ;
\path \EastDeux node[midway,right] {$\Smac_{0,d}$} ;
\path (0,0) -- (0, \taille) node[midway, left] {$\swbsOmega_d(g) \ = \ $} ;
\end{tikzpicture}
}

Bien sûr, comme pour tout endomorphisme en mathématique, 
on n'est pas contraint par l'ordre à mettre sur la base. 
Ici, on peut choisir de prendre une base ordonnée quelconque de $\bfA[\uX]_d$.
La seule contrainte est de respecter la définition (ah-ah), c'est-à-dire :
$$
\begin{tikzpicture}[scale = 0.8] 
\draw [thick] (-0.1, -0.3) -- (-0.4, -0.3) -- (-0.4, 6.3) -- (-0.1, 6.3) ;
\draw [thick] (6+0.1, -0.3) -- (6+0.4, -0.3) -- (6+0.4, 6.3) -- (6+0.1, 6.3) ;
\draw [dashed, rounded corners=4pt, fill=gray!20] (1.8, -0.3) rectangle (2.2, 6.3) ;
\draw (2,6.3) node[above] {\footnotesize $\dfrac{X^\alpha}{X_i^{d_i}} P_i$} ;
\draw [dotted] (2,4) -- (6.6,4) node[right] {\footnotesize $X^\alpha \in \Jex_{1,d}$} ;
\draw [dashed, rounded corners=4pt, fill=gray!60] (4.8,-0.3) rectangle (5.2,6.3) ;
\draw (5,6.3) node[above] {$g(X^{\emouton-\gamma}) \atop = g^\sw(X^\gamma)$} ;
\draw [dotted] (5,1) -- (6.6,1) node[right] {\footnotesize $X^{\gamma} \in \Smac_{0,d}$} ;
\foreach \r in {0,1,...,6} \draw (\r, 6-\r) node {\tiny $\bullet$} ;
\end{tikzpicture}
$$
}

On peut modifier très légèrement la vision de l'ensemble de définition de
$\swbsOmega_d$.  On peut en effet décréter
que $\swbsOmega_d$ prend comme argument une \textit{famille
  $\bsnabla_{d'}$ indexée par la base monomiale de~$\Smac_{0,d'}$}.
Supposons de plus $d' < \min(D)$. Cette inégalité a comme conséquence
que $\Smac_{0,d'} = \bfA[\uX]_{d'}$, donc $\swbsOmega_d$ prend comme
argument une famille $\bsnabla_{d'} = (\nabla_\beta)_{|\beta|= d'}$ de polynômes  de
$\bfA[\uX]_d$ indexée par les $\beta$ tels que $|\beta| = d'$. D'où la
définition suivante :
$$
\forall\, |\alpha| = d, \qquad
\swbsOmega_d(\bsnabla_{d'}) : \ X^\alpha \ \longmapsto \ 
\begin {cases}
\nabla_{\emouton-\alpha}                  &\text{si $X^\alpha \in \Smac_{0,d}$} \\[3mm]
\dfrac{X^\alpha}{X_i^{d_i}}P_i &\text{si $X^\alpha \in \Jex_{1,d}$ où $i = \minDiv(X^\alpha)$} \\
\end {cases}
$$

\label {NOTA16-swbsOmegad}%
\label {NOTA16-swbsOmegasigmad}%
%
%

\subsection {Le résultant via la forme $r_{0,d}$-linéaire alternée $\swDet_{d,\protect\uP} : \BdSodprime\to\bfA$}

\begin {theo}
\label {ResViaDetd}
En rappelant que $\delta - \min(D) < d \leqslant \delta$ et $d' = \delta-d$, le résultant s'obtient
comme l'évaluation de la forme $r_{0,d}$-linéaire alternée $\swDet_d$ en la
famille $\bsnabla_{d'} = (\nabla_\beta)_{|\beta| = d'}$ des déterminants de Sylvester:
$$
\Res(\uP) = \swDet_{d,\uP}\big(\bsnabla_{d'}\big)
$$
Ce qui, par définition, équivaut au quotient de deux déterminants:
$$
\Res(\uP) = \dfrac{\det \swbsOmega_d\big((\nabla_\beta)_\beta\big)}{\Delta_{2,d}}
= \dfrac{\det \swbsOmega_d\big((\nabla_\beta)_\beta\big)}{\det W_{2,d}}
$$
On dispose également d'une version tordue par une permutation $\sigma \in \fS_n$:
$$
\Res(\uP) = 
\dfrac{\det\bsOmega^{\sigma\,\sw}_d\big((\nabla_\beta)_{\beta}\big)}{\det W^\sigma_{2,d}}
$$
\end {theo}

\medskip

Pour $d = \delta$, on retrouve l'égalité $\Res(\uP) =
\omegares(\nabla)$ mais le résultat ci-dessus ne constitue pas une
nouvelle preuve pour la bonne raison qu'il a bien fallu donner une
définition du résultant, définition qui s'appuie justement sur
$\Res(\uP) = \omegares(\nabla)$.

\begin {proof} \leavevmode

On travaille en générique en désignant comme d'habitude par $\bfA
= \bfk[\indetsPi]$ l'anneau des coefficients du système $\uP$. Notons
$\calR$ l'évaluation $\swDet_{d,\uP}\big((\nabla_\beta)_{|\beta|=d'}\big)$. Chaque
$\nabla_\beta$ étant dans $\uPsat_d$, le théorème~\ref{DetdIsCramer}
fournit $\calR \in \ElimIdeal$, donc $\calR$ est multiple de
$\Res(\uP)$.  On va étudier le poids en chaque $P_i$ de $\calR$ et la
spécialisation de $\calR$ en le jeu étalon pour en déduire $\calR
= \Res(\uP)$.

\medskip
$\bullet$
\'Etude du poids. Ici on peut faire relâche avec le micmac sw et étudier $\Det_d(\sbullet)$ par l'intermédiaire
du quotient $\det \bsOmega_d(\sbullet) / \det W_{2,d}$.

Soit $\calF \in \AXdSod$ une famille quelconque. Le développement de
Laplace de $\det \bsOmega_d(\calF)$ par rapport à l'ensemble $J_0$ des
colonnes indexées par les monômes de $\Jex_{1,d}$ est une somme de
termes de la forme
$$
\pm \det \bsOmega_d(\calF)_{I \times J_0} \times \det \bsOmega_d(\calF)_{\overline I \times \overline {J_0}} 
$$
Le mineur de gauche, d'ordre $r_{1,d} = \dim\Jex_{1,d}$, ne dépend que de $\uP$ tandis que celui
de droite, d'ordre~$r_{0,d}$, ne dépend que de $\calF$ (cf. les exemples en
section \ref{CanonicalCayleyDetSection} et les deux dessins de la page~\pageref{DessinOmegad}).
De manière un peu plus précise, $\det\bsOmega_d(\calF)$ est une somme signée de termes du type:
$$
\text{mineur d'ordre $r_{1,d}$ de $\Syl_d(\uP)_{|\Smac_{1,d}}$} 
\ \times \ 
\text{mineur d'ordre $r_{0,d}$ de la matrice de colonnes $\calF$}
$$
Donc $\Det_d(\calF)$ est une somme signée de termes du type:
$$
\dfrac{\text{mineur d'ordre $r_{1,d}$ de $\Syl_d(\uP)_{|\Smac_{1,d}}$}}{\det W_{2,d}}
\ \times \ 
\text{mineur d'ordre $r_{0,d}$ de la matrice de colonnes $\calF$}
$$
A gauche, le numérateur est homogène en $P_i$ de poids $\dim\Jex_{1,d}^{(i)}$. Idem
pour le dénominateur  avec comme poids $\dim\Jex_{2,d}^{(i)}$. De sorte que
le quotient (exact) est homogène en $P_i$ de poids $\dim\Jex_{1\setminus2,d}^{(i)}$.
Particularisons maintenant $\calF$ à la famille $(\nabla_\beta)_{|\beta|=d'}$
(sans nous soucier du mouton-swap). Chaque colonne $\nabla_\beta$ est de poids~1
en $P_i$ si bien que le mineur de droite est homogène en $P_i$ de poids $r_{0,d}$.

\smallskip
Chaque produit ci-dessus est donc homogène en $P_i$ de poids $\dim \Jex_{1\setminus2,d}^{(i)} + r_{0,d}$.
Or ce nombre, d'après le lemme ci-après, est égal à $\widehat d_i$. 

\medskip
Bilan: l'évaluation $\calR$ est donc homogène en $P_i$ de poids $\widehat d_i$.
Par conséquent, $\calR = \lambda\,\Res(\uP)$ où $\lambda \in \bfA$ est
homogène de poids $0$, donc un élément du petit anneau $\bfk$.

\medskip
$\bullet$ Détermination de $\lambda$.
Rappelons qu'en désignant par $\iota : \calS_{0,d} \hookrightarrow \bfA[\uX]_d$ l'injection
canonique \mbox{$X^\alpha \mapsto X^\alpha$}, on
a $\Det_{d,\uP}(\iota) = \Delta_{1,d}(\uP)$. Ce qui fournit
$$
\swDet_{d,\uP}\big( (X^{\emouton-\beta})_{|\beta| = d'} \big) = \Delta_{1,d}(\uP)
$$
Or on peut faire en sorte d'avoir choisi les matrices $\beta$-bezoutiennes de $\uP$
de manière à ce que $\nabla_\beta$ se spécialise en $X^{\emouton-\beta}$ lorsque l'on
spécialise $\uP$ en le jeu étalon $\uX^D$.
En conséquence, la spécialisation de $\swDet_{d,\uP}\big((\nabla_\beta)_{|\beta|=d'}\big)$
en le jeu étalon fournit $\Delta_{1,d}(\uX^D)$, qui est égal à~1.
Bilan: $\lambda = 1$ et la démonstration est terminée.

\end {proof}

\begin{lem} \label{PoidsJi12dS0d}

Notons $\bfA[\widehat{X_i}] = \bfA[X_1,\dots,X_{i-1},\, X_{i+1}, \dots, X_n] \subset
\bfA[\uX]$.

\begin{enumerate}[\rm i)]
\item
Le $\bfA$-module $\Smac_0 \cap \bfA[\widehat{X_i}]$ 
a pour base les~$X^\gamma$ tels que $\gamma \preccurlyeq \emouton$ et $\gamma_i = 0$.
Sa dimension est $\widehat{d}_i$.

\item  
Sous le couvert des inégalités $\delta - \min(D) < d \leqslant \delta$,
le $\bfA$-module monomial $\Jex_{1\setminus 2,d}^{(i)} \oplus \Smac_{0,d}$ 
est isomorphe au $\bfA$-module $\Smac_0 \cap \bfA[\widehat{X_i}]$.
En particulier:
$$
\dim \Jex_{1\setminus 2,d}^{(i)} 
\ + \ 
r_{0,d}
\ = \ 
\widehat{d}_i
$$
\end{enumerate}

\end{lem}

\begin{proof} \leavevmode

\medskip
i) Le nombre de $\gamma \in \bbN^n$ tels que $\gamma \preccurlyeq \emouton$ et 
$\gamma_i = 0$ (c'est-à-dire $\gamma_j < d_j$ pour tout $j \neq i$ et $\gamma_i = 0$) 
est égal à $\widehat d_i = \prod_{j\neq i} d_j$, 
donc $\dim \big(\Smac_0 \cap \bfA[\widehat{X_i}]\big) = \widehat d_i$. 
  
\medskip  
ii)
On va vérifier que les deux applications ci-dessous sont bien définies.
Le fait qu'elles soient réciproques l'une de l'autre est laissé au lecteur.
$$
\begin{array}[t]{rcl}
\Jex_{1\setminus 2,d}^{(i)} \oplus \Smac_{0,d} 
& \longrightarrow & \Smac_0 \cap \bfA[\widehat{X_i}] \\ [2mm]
X^\alpha & \longmapsto & \dfrac{X^\alpha}{X_i^{\alpha_i}}
\end{array}
\qquad \qquad \text{et} \qquad \qquad 
\begin{array}[t]{rcl}
\Smac_0 \cap \bfA[\widehat{X_i}] & \longrightarrow & 
\Jex_{1\setminus 2,d}^{(i)} \oplus \Smac_{0,d} \\ [2mm]
X^\gamma & \longmapsto & X_i^{d - |\gamma|} X^\gamma 
\end{array}
$$
Il est clair que $\dfrac{X^\alpha}{X_i^{\alpha_i}}$, qui est dans
$\bfA[\widehat{X_i}]$, est aussi dans $\Smac_0$ (que $X^\alpha$ soit
dans $\Jex_{1\setminus 2}^{(i)}$ ou dans~$\Smac_0$).  Pour
l'application de droite, il faut d'abord s'assurer que la formule
donnée a du sens, autrement dit que $d - |\gamma | \geqslant 0$, ou
encore que le degré maximum d'un élément de
$\Smac_0 \cap \bfA[\widehat{X_i}]$ est inférieur à $d$.  Or ce degré
maximum vaut $\delta - (d_i -1)$ qui est bien $\leqslant d$ grâce à
la contrainte $\delta - \min(D) < d$.  Le fait que $X^\alpha =
X_i^{d-|\gamma|}X^\gamma$ est bien dans le module d'arrivée utilise:
$$
\forall\, j \neq i, \quad \alpha_j = \gamma_j \ < \ d_j 
\qquad \text{ et} \qquad
\alpha_i = d - |\gamma| 
$$
de sorte que $X^\alpha$ appartient à $\Jex^{(i)}_{1\setminus2,d}$ si $\alpha_i \geqslant d_i$
et à $\Smac_{0,d}$ si $\alpha_i < d_i$.
\end{proof}

\begin{prop}
La formule via $\Det_d$ est une formule déterminantale (\idest{} sans dénominateur)
pour le résultant si et seulement si 
$$
\delta - \min(D) \ < \ 
d 
\ < \ \min_{i\neq j} (d_i + d_j)
$$
Supposons $d_1 \leqslant d_2 \leqslant \cdots \leqslant d_n$.
L'intervalle ci-dessus est non vide si et seulement si 
$$
d_3 + \cdots + d_n \ \leqslant \ d_1 + n - 2
\leqno (\star)
$$
Pour $n=2$, tout format $D$ convient.

\smallskip 
\noindent
Pour $n=3$, 
le format de degrés est du type 
$$
D = (d_1, d_1, d_1) \qquad 
D = (d_1, d_1, d_1+1) \qquad 
D = (d_1, d_1+1, d_1+1)
$$
Pour chaque $n \geqslant 4$, il y a deux formats du type 
$D = (1,\dots,1,d_n)$ avec $d_n \leqslant 2$.

\smallskip
\noindent
Pour $n=4$, on a aussi $D = (2,2,2,2)$.
\end{prop}

\begin{proof}
Au dénominateur de $\Det_d$ intervient $\det W_{2,d}$.
En générique, ce terme vaut $1$ si et seulement si $\Jex_{2,d} = 0$, 
ce qui équivaut d'après~\ref{NulliteJhd} à $d < \min_{i\neq j} (d_i + d_j)$.
Si $d_1 \leqslant d_2 \leqslant \cdots \leqslant d_n$, la condition s'écrit 
$\delta - d_1 < d < d_1+d_2$, c'est-à-dire 
$$
d_2 + \cdots + d_n -n < d < d_1 + d_2
$$
Par conséquent, demander l'existence d'une formule déterminantale  
revient à imposer l'inégalité $(\star)$.

\medskip

\noindent
$\rhd$ Pour $n=2$, l'inégalité $(\star)$ s'écrit $0 \leqslant d_1$, donc il n'y a aucune condition.

\noindent
$\rhd$ Pour $n=3$, on a $d_3 \leqslant d_1 + 1$, d'où les trois types évoqués.

\medskip
\noindent
Dans l'inégalité $(\star)$, les entiers de gauche sont minorés
par $d_3$ et, à droite, $d_1$ est majoré par~$d_3$, d'où $(n-2)
d_3 \leqslant d_3 + n-2$ c'est-à-dire $(n-3) d_3 \leqslant n-2$.  Pour
$n \geqslant 4$, on a donc $d_3 \leqslant \frac{n-2}{n-3}$. Ainsi,
\begin{center}
pour $n=4$, \ $d_3 \leqslant 2$ ; \qquad 
pour $n \geqslant 5$, \ $d_3 \leqslant 1$.
\end{center} 

\noindent
$\rhd$ Pour $n = 4$, reste à traiter le cas $d_3 = 2$.
L'inégalité $(\star)$ s'écrit $d_3 + d_4 \leqslant d_1 + 2$, d'où $d_4 \leqslant d_1$. 
Ainsi $d_1=d_4$ et comme $d_3 = 2$, on obtient $D=(2,2,2,2)$.

\end{proof}

\subsection {Relations entre les formes $\Det_{d,\protect\uP}$, $\omegaRes{\protect\uP}$
    et les déterminants de Sylvester $(\nabla_\beta)_{|\beta|=d'}$}

Le théorème \ref {ResViaDetd} affirme que le résultant de $\uP$
s'obtient comme une évaluation de la forme $\swDet_d : \BdSodprime \to \bfA$
en la famille $(\nabla_\beta)_{|\beta| = d'}$. Dans cette
section, nous allons exprimer d'autres évaluations spécifiques de la forme
$r_{0,d}$-linéaire alternée $\swDet_d$ sur~$\bfB_d$ à l'aide
d'évaluations de la forme $\omegares$ sur $\bfB_\delta$.  Ce résultat
ne suppose aucune hypothèse sur $\uP$ puisque les expressions
en question sont des identités algébriques.

Ceci nous permettra d'obtenir une première preuve du fait que la
famille $(\nabla_\beta)_{|\beta|=d'}$ est une $\bfA$-base de
$\vH^0_\uX(\bfB)_d = \uPsat_d /\langle\uP\rangle_d$ lorsque~$\uP$ est
régulière. Nous insistons ici sur le fait que ce résultat ne
suppose aucunement $\uP$ générique (contrairement à \cite {J7} ou
au résumé réalisé dans la section 2 de \cite{BuseChardinNemati}) et il n'y
a aucune raison pour que l'idéal d'élimination de $\uP$ soit engendré
par~$\Res(\uP)$.
Cette approche, en particulier basée sur la propriété \og
Cramer\fg{} de la forme $\Det_d$, sera revisitée en y intégrant
la dualité de portée plus générale:
$$
\vH^0_\uX(\bfB)_d  \simeq (\bfB_{\delta-d})^\star
$$
En conservant le contexte $\delta-\min(D) < d \leqslant \delta$ et $d' = \delta-d$, nous devons
préciser ce que nous entendons par évaluations spécifiques de $\swDet_d$. Grosso-modo,
pour $|\beta| = d'$ et $F \in \bfB_d$, il s'agit dans la famille 
$(\nabla_{\beta'})_{|\beta'|=d'}$ de remplacer la composante d'indice $\beta$
par $F$. Ce qui nous amène à définir de manière formelle la fonction
$\Delta(\beta,F) \in \BdSodprime$ de la façon suivante:
$$
\begin {array}{cccl}
\Delta(\beta,F) : &\calS_{0,d'}& \to &\bfB_d  \\[2mm]
                 &X^{\beta'}& \longmapsto &
\begin {cases} \nabla_{\beta'} &\text{si $\beta'\ne\beta$}\\ F&\text{sinon} \\ \end {cases}  
\end {array}
$$
Prenons par exemple $F = \nabla_\gamma$ avec $|\gamma| = d'$. Si $\gamma\ne\beta$, alors
$\Delta(\beta,F)$ a deux composantes égales, donc $\swDet_d(\Delta(\beta,F)) = 0$. Et
si $\gamma = \beta$, alors $\Delta(\beta,F)$ n'est autre que $(\nabla_{\beta'})_{|\beta'|=d'}$
donc $\swDet_d(\Delta(\beta,F))$ est égal à $\Res(\uP) \overset{\rm aussi}{=} \omegares(\nabla)$.
Lecteur, tu suis?
Ceci permet d'énoncer le théorème principal.

\begin {theo} \leavevmode
\label{omegaresDetdRelations}  

Soient $X^\beta$ de degré $d'$ et $F \in \bfB_d$.

\begin{enumerate} [\rm i)]
\item
Les formes $\Det_d$ et $\omegares$ sont reliées par l'égalité des évaluations
suivantes:  
$$
\swDet_d\big(\Delta(\beta,F)\big) = \omegares(X^\beta F)
$$  

\item
On dispose de l'égalité (dans $\bfB_d$):
$$
\omegares(\nabla)\,F = \sum_{|\beta|=d'} \omegares(X^\beta F)\,\nabla_\beta  
$$

\item
Lorsque $\uP$ est régulière, la famille $(\nabla_\beta)_{|\beta|=d'}$ est une
$\bfA$-base de $\vH^0_\uX(\bfB)_d = \uPsat_d /\langle\uP\rangle_d$. 
\end{enumerate}    
\end {theo}

De manière à atténuer le côté hermétique de l'énoncé, nous avons
choisi de donner la preuve à l'aide d'un exemple significatif.
Nous donnons également une seconde preuve du point ii) indépendante de
la forme $\Det_d$, preuve reportée après celle du théorème d'orthogonalité
qui suit.


\medskip

Nous pouvons donner une interprétation matricielle du point ii)
via l'égalité d'une matrice scalaire comme produit de deux matrices
rectangulaires.  La matrice de gauche a pour lignes l'ensemble des
monômes $X^\beta \in \bfA[\uX]_{d'}$, et pour colonnes l'ensemble des
monômes $X^\alpha \in \bfA[\uX]_{d}$.  Et à droite, c'est inversé.
Quant à la notation $(\nabla_{\beta'})_\alpha$, elle désigne la
composante sur $X^\alpha$ du polynôme $\nabla_{\beta'}$ homogène de
degré $d$.
$$
\begin {bmatrix}
  &    &                            &   & \\
  &    &\omegares(X^{\beta + \alpha})_{\substack{|\beta|=d'\\|\alpha|=d}} &   & \\
  &    &                            &   & \\
\end {bmatrix}
\begin {bmatrix}
  &                   & \\
  &                   & \\
  &\big((\nabla_{\beta'})_\alpha\big)_{\substack{|\alpha|=d\\|\beta'|=d'}}    & \\    
  &                   & \\
  &                   & \\
\end {bmatrix}
=
\omegares(\nabla)\,\Id_{\bfA[\uX]_{d'}}
$$

\subsubsection*{\og Orthogonalité\fg{} des deux familles $(X^\gamma)_{|\gamma|=d'}$ 
et $(\nabla_\beta(\uP))_{|\beta|=d'}$}

Le résultat suivant figure chez Jean-Pierre Jouanolou en
\cite[proposition 3.10.6, pages 161-162] {J7}.
Nous en proposons une preuve élémentaire (la première preuve).

\begin {theo} [Orthogonalité]
\label{DualityJPJ3.10.6}

Soit $\nabla$ ``le'' déterminant bezoutien de $\uP$. Alors  
pour tous \mbox{$\beta,\gamma\in\bbN^n$} tels que $|\beta| = |\gamma| < \min(D)$
et tout déterminant de Sylvester $\nabla_\beta$ de $\uP$:
$$
X^\gamma \nabla_\beta = \begin {cases}
\nabla &\text {si $\beta = \gamma$} \\
0      &\text{sinon} \\  
\end {cases}
\qquad   \text{dans $\vH^0_\uX(\bfB)_\delta$}
$$
En des termes équivalents, puisque $\vH^0_\uX(\bfB)_\delta = \uPsat_\delta /\langle\uP\rangle_\delta$,
on a des congruences modulo $\langle\uP\rangle_\delta$:
$$
X^\gamma \nabla_\beta \equiv \begin {cases}
\nabla &\text {si $\beta = \gamma$} \\
0      &\text{sinon} \\  
\end {cases}
$$
\end {theo}

\noindent
{\bf Commentaire}

Pour établir ce résultat, on peut supposer $\uP$ générique auquel cas
la forme $\uomegares : \bfB_\delta \to \bfA$ est injective. Dans ce
cadre, l'énoncé est donc équivalent à l'égalité
$\omegares(X^\gamma\,\nabla_\beta) = \omegares(\nabla) \text{ ou } 0$
selon $\beta = \gamma$ ou $\beta \ne \gamma$.

Ainsi, ce résultat d'orthogonalité est étroitement lié à l'égalité du
point i) du théorème \ref{omegaresDetdRelations}. En effet, examinons
ce point i) en y prenant $F = \nabla_\gamma$. 
Comme 
$$
\Delta(\beta, \nabla_\gamma)\ 
\begin {cases}
\text{vaut } (\nabla_{\beta'})_{|\beta'| = d'} &\text {si $\beta = \gamma$} \\
\text{possède 2 composantes égales}  &\text{sinon} \\  
\end {cases}
$$
on obtient 
$$
\omegares(X^\beta\, \nabla_\gamma) = 
\swDet_d\big(\Delta(\beta, \nabla_\gamma)\big) 
= \begin {cases}
\swDet_d\big((\nabla_{\beta'})_{|\beta'|=d'}\big) &\text {si $\beta = \gamma$} \\
0     &\text{sinon} \\  
\end {cases}
$$
Vu la définition de l'évaluation de $\swDet_d$, cette expression est encore égale à 
$\omegares(\nabla)$ si $\beta = \gamma$ ou bien $0$ sinon.  C'est
exactement le résultat d'orthogonalité ci-dessus (à la permutation
$\beta \leftrightarrow \gamma$ près).

\begin {proof} [Première preuve (élémentaire)] \leavevmode

$\rhd$ Le cas $\gamma = \beta$.
Puisque deux déterminants bezoutiens de $\uP$ sont égaux modulo
$\langle\uP\rangle_\delta$, il suffit de montrer que $X^\beta\nabla_\beta$ est un
déterminant bezoutien de $\uP$. Avec des notations qui se devinent, on a:
$$
[P_1, \dots, P_n] = [X_1^{\beta_1+1}, \dots,  X_n^{\beta_n+1}]\,\dsV_\beta =
[X_1, \dots, X_n]\,\dsV  \qquad \text{où} \qquad
\dsV = \diag(X_1^{\beta_1},  \dots, X_n^{\beta_n})\,\dsV_\beta     
$$
Donc $\dsV$ est une matrice bezoutienne de $\uP$, d'où le résultat puisque
$\det\dsV$ vaut $X^\beta \det \dsV_\beta = X^\beta \nabla_\beta$.
  
\medskip  
$\rhd$
Le cas $\gamma \ne \beta$. Multiplions l'égalité  $\uP = \uX^{\beta+\Un}.\dsV_\beta$ à droite
par la cotransposée de $\dsV_\beta$:
$$
\uP.\widetilde {\dsV_\beta} = \nabla_\beta . \uX^{\beta+\Un}
$$
En conséquence, $X_i^{\beta_i + 1} \nabla_\beta \in \langle\uP\rangle$
pour chaque $i$. Or $|\gamma| = |\beta|$ et $\gamma \ne \beta$ donc il
existe un $i$ tel que $\gamma_i > \beta_i$. On en déduit
$X_i^{\gamma_i} \nabla_\beta \in \langle\uP\rangle$, a fortiori
$X^\gamma \nabla_\beta \in \langle\uP\rangle$.
\end {proof}

A titre indicatif, voici une seconde preuve, à quelque chose près
celle figurant en \cite{J7}. On y voit un détail qui semble insignifiant
mais qui est en fait capital: c'est l'existence d'un $i$ tel que
$\gamma_i > \beta_i$.

\begin {proof} [Seconde preuve]\leavevmode

  
Le cas $\gamma \ne \beta$. On peut supposer $\uP$ générique et faire
en sorte de disposer d'une matrice $\beta$-bezoutienne $\dsV_\beta$
pour $\uP$ qui se spécialise en $\diag(\uX^{\emouton - \beta})$
lorsque l'on spécialise~$\uP$ en le jeu étalon~$\uX^D$.  D'après le
cas $\gamma=\beta$ de la preuve précédente, on peut supposer $\nabla = X^\beta \nabla_\beta$ où
$\nabla_\beta = \det(\dsV_\beta)$.  Ainsi les spécialisations
respectives de $\nabla_\beta$ et $\nabla$ en le jeu étalon sont:
$$
X^{\emouton-\beta}, \qquad  X^{\emouton}
$$
Le polynôme $X^\gamma \nabla_\beta$ de degré $\delta$ est dans $\uPsat$.
Comme $\uPsat_\delta = \langle\uP\rangle_\delta \oplus \bfA\,\nabla$, 
il existe une constante homogène $c \in \bfA = \bfk[\indetsPi]$ telle que
$$
X^\gamma \nabla_\beta \equiv c\,\nabla \bmod \langle \uP\rangle
$$
On va montrer que $c=0$. Le poids en chaque $P_i$ de $X^\gamma \nabla_\beta$
et $\nabla$ étant~1, il en résulte que $c \in \bfk$.
En spécialisant la congruence ci-dessus en le jeu étalon, on obtient
$$
X^{\emouton+\gamma-\beta} \equiv c\,X^\emouton \bmod \langle\uX^D\rangle
$$
Comme $|\gamma| = |\beta|$  et $\gamma \ne \beta$, il existe un
$i$ tel que $\gamma_i > \beta_i$. Donc $X^{\emouton+\gamma-\beta} \in \langle\uX^D\rangle$.
Il vient $c\,X^\emouton \in \langle\uX^D\rangle$ d'où $c=0$.
\end {proof}

\begin {proof}[Première preuve du théorème \ref{omegaresDetdRelations} par l'exemple] \leavevmode

Soit $\nabla$ la classe dans $\bfB_\delta$ du déterminant bezoutien
de $\uP$.  Dans cet exemple qui se veut significatif, nous supposons
$r_{0,d'} = 3$. La base monomiale de $\bfA[\uX]_{d'}$ est donc
constituée de 3 monômes $X^\alpha, X^\beta, X^\gamma$ de degré $d'$,
en correspondance \og mouton-swap\fg{} avec la base monomiale
$\calS_{0,d}$ de $\Smac_{0,d}$, qui fournissent 3 déterminants de
Sylvester $\nabla_\alpha, \nabla_\beta, \nabla_\gamma$ (de degré~$d$)
de $\uP$. Nous notons les arguments de $\Det_d : \bfB_d^{\calS_{0,d}} \to \bfA$
en respectant cet ordre.

\medskip
$\bullet$
Soit $F \in \bfB_d$. La propriété \og Cramer\fg{} de $\Det_d$ s'écrit
(sur $\bfB_d$) de la manière suivante:
$$
\setlength{\arraycolsep}{0.5\arraycolsep}
\Det_d(\nabla_\alpha, \nabla_\beta, \nabla_\gamma)\,F =
a_\alpha\nabla_\alpha +  a_\beta\nabla_\beta +  a_\gamma\nabla_\gamma
\qquad \text{où} \qquad
\left\{
\begin {array} {ccl}
a_\alpha &=& \Det_d(F, \nabla_\beta, \nabla_\gamma) \\[2mm]
a_\beta &=& \Det_d(\nabla_\alpha, F, \nabla_\gamma) \\[2mm]
a_\gamma &=& \Det_d(\nabla_\alpha, \nabla_\beta, F) 
\end {array}
\right.
\leqno(\star)
$$
En utilisant le théorème \ref{ResViaDetd}, nous pouvons ré-écrire cette égalité:
$$
\omegares(\nabla)\,F = a_\alpha\nabla_\alpha +  a_\beta\nabla_\beta +  a_\gamma\nabla_\gamma
$$
En multipliant par $X^\alpha$ l'égalité ci-dessus et en utilisant le
théorème~\ref{DualityJPJ3.10.6}, nous obtenons (dans $\bfB_\delta$):
$$
\omegares(\nabla)\,X^\alpha F = a_\alpha\nabla
$$
et deux autres égalités analogues en remplaçant $X^\alpha$ par $X^\beta$ et $X^\gamma$:
$$
\omegares(\nabla)\,X^\beta F = a_\beta\nabla,\qquad    \omegares(\nabla)\,X^\gamma F = a_\gamma\nabla
\qquad \text{dans $\bfB_\delta$}
$$
Appliquons $\omegares$ aux trois égalités obtenues dans $\bfB_\delta$. Quitte
à mettre un peu de généricité assurant le fait que $\omegares(\nabla)$ est
régulier, nous pouvons simplifier par $\omegares(\nabla)$ pour obtenir
les 3 égalités scalaires:
$$
\omegares(X^\alpha F) = a_\alpha, \qquad
\omegares(X^\beta F) = a_\beta, \qquad
\omegares(X^\gamma F) = a_\gamma 
$$
Nous avons ainsi relié les évaluations de $\Det_d$ et celles de $\omegares$:
$$
\Det_d(F, \nabla_\beta, \nabla_\gamma) = \omegares(X^\alpha F),\qquad
\Det_d(\nabla_\alpha, F, \nabla_\gamma) = \omegares(X^\beta F), \qquad
\Det_d(\nabla_\alpha, \nabla_\beta, F) = \omegares(X^\gamma F) 
$$
Ceci fournit, avec l'égalité $(\star)$, la preuve des points i) et ii).

\medskip
$\bullet$
Imposons désormais à $\uP$ d'être une suite régulière de sorte que $\omegares(\nabla)$ est
un scalaire régulier.

\smallskip
$\rhd$ Aspect générateur.
Soit $F \in \uPsat_d /\langle\uP\rangle_d$ que nous voulons exprimer comme
combinaison $\bfA$-linéaire de $\nabla_\alpha, \nabla_\beta, \nabla_\gamma$.
On a $X^\alpha F \in \uPsat_\delta /\langle\uP\rangle_\delta = \bfA\,\nabla$, donc
$\omegares(X^\alpha F) \in \bfA\,\omegares(\nabla)$. Idem pour les deux autres,
d'où l'expression (dans laquelle les quotients sont exacts):
$$
F = \frac{\omegares(X^\alpha F)}{\omegares(\nabla)}\,\nabla_\alpha  +
\frac{\omegares(X^\beta F)}{\omegares(\nabla)}\,\nabla_\beta +
\frac{\omegares(X^\gamma F)}{\omegares(\nabla)}\,\nabla_\gamma
$$

\smallskip
$\rhd$ Liberté. Soit une relation $\bfA$-linéaire dans $\bfB_d$:
$$
\lambda_\alpha\,\nabla_\alpha +  \lambda_\beta\,\nabla_\beta +  \lambda_\gamma\,\nabla_\gamma = 0
$$
En évaluant $\Det_d(\sbullet, \nabla_\beta, \nabla_\gamma)$ sur cette égalité, on obtient
$\lambda_\alpha\omegares(\nabla) = 0$ d'où $\lambda_\alpha = 0$. Idem pour les deux autres.
Ceci termine la preuve du point iii).
\end {proof}

\begin {proof} [Preuve alternative indépendante des points ii) et iii) de \ref{omegaresDetdRelations}]
\leavevmode

On peut se passer complètement de la forme $\Det_d$ pour établir ces deux points mais bien
entendu, procédant ainsi, on ne relie pas $\Det_d$ et $\omegares$.

\medskip

ii) Comme c'est une identité algébrique, on peut supposer $\uP$
régulière. Notons $L$ le membre gauche et $R$ celui de droite. Pour
montrer $L=R$ dans~$\bfB_d$, il suffit d'après l'aspect \og
non-dégénéré\fg{} de la multiplication
(proposition~\ref{SemiPairingProperty}) de montrer que pour tout
monôme $X^\gamma$ de degré $d'=\delta-d$, on a $X^\gamma L = X^\gamma R$
dans $\bfB_\delta$.  En vertu du théorème d'orthogonalité \ref{DualityJPJ3.10.6}, on a
dans~$\bfB_\delta$:
$$
X^\gamma R \overset{\rm def}{=} \sum_{|\beta| = d'} \omegares(X^\beta F)\,X^\gamma \nabla_\beta =
\omegares(X^\gamma F) \nabla
$$
Versus $X^\gamma L \overset{\rm def}{=} \omegares(\nabla)\, X^\gamma F$.
L'égalité résulte alors de la propriété Cramer de $\omegares$.

\medskip
iii) On procède comme à la fin de la première preuve en utilisant, pour l'aspect
générateur, l'identité du point~ii). En ce qui concerne la liberté, considérons
une relation $\bfA$-linéaire dans $\bfB_d$:
$$
\sum_{|\beta|= d'} \lambda_\beta\, \nabla_\beta = 0
$$
En multipliant par $X^\gamma$ où $|\gamma| = d'$ et en utilisant le
résultat d'orthogonalité, on obtient $\lambda_\gamma\,\nabla= 0$
(dans~$\bfB_\delta$) donc $\lambda_\gamma = 0$.
\end {proof}

\subsection {La base de $\vH^0_{\protect\uX}(\bfB)_d$ réexaminée dans le cadre de la dualité}
\label{HCech0B}

Comme nous sommes très attachés aux problèmes de fondement et à
la problématique oeuf/poule, nous revenons encore une fois sur la preuve du fait
que $(\nabla_\beta)_{|\beta| = d'}$ est une $\bfA$-base de
$\vH^0_\uX(\bfB)_d$ lorsque $\uP$ est régulière.  Ici $\delta
- \min(D) < d \leqslant \delta$, mais nous libérerons $d$ un peu plus loin.
Dans la preuve qui vient, nul besoin ni de $\omegares$ ni de $\Det_d$
(cependant $\omegares$ pointe le bout de son nez), l'ingrédient
principal étant le résultat d'orthogonalité.

\medskip

Soit $F \in \vH^0_\uX(\bfB)_d = \uPsat_d/\langle\uP\rangle_d$ à écrire
comme combinaison $\bfA$-linéaire des $(\nabla_\beta)_{|\beta| =
d'}$. Comme $X^\beta F \in \vH^0_\uX(\bfB)_\delta = \bfA\nabla$, il existe un unique $a_\beta \in \bfA$
tel que $X^\beta F = a_\beta \nabla$. Nous affirmons alors que dans~$\bfB_d$:
$$
F = \sum_{|\beta|=d'} a_\beta\,\nabla_\beta
$$
En effet, en notant $G$ la somme de droite, pour tout $|\gamma| = d'$,
on a, d'après le théorème d'orthogonalité, $X^\gamma F = X^\gamma G$
dans $\bfB_\delta$ donc $F = G$ d'après l'aspect \og non-dégénéré\fg{} de la multiplication,
cf. la proposition~\ref{SemiPairingProperty}.
Quant à la liberté de la famille $(\nabla_\beta)_{|\beta| = d'}$, elle est conséquence
de l'orthogonalité.

\smallskip

Il faut noter cependant que nous n'obtenons pas l'identité du point ii) du
théorème~\ref{omegaresDetdRelations} car ici $F \in \vH^0_\uX(\bfB)_d$
tandis que dans le théorème $F \in \bfB_d$.
Un lecteur attentif remarquera que $\omegares$ n'est pas loin puisque
$a_\beta$ est le quotient exact $\omegares(X^\beta F)/\omegares(\nabla)$.

\bigskip

Il est important de remarquer, puisque $d' := \delta-d < \min(D)$, que
$\langle \uP\rangle_{d'} = 0$, donc $\bfB_{d'} = \bfA[\uX]_{d'}$ est
un $\bfA$-module libre avec sa base monomiale $(X^\gamma)_{|\gamma| =
d'}$.  Nous allons maintenant revenir sur l'isomorphisme $\vH^0_\uX(\bfB)_d \simeq
(\bfB_{d'})^\star$ et constater que la famille
$(\nabla_\beta)_{|\beta| = d'}$ est la base duale de cette base
monomiale de~$\bfB_{d'}$.

\subsubsection*{Dualité et l'isomorphisme $\vH^0_\uX(\bfB)_d \simeq (\bfB_{d'})^\star$
                où ici $0 \leqslant d \leqslant \delta$ et $d' = \delta-d$}

Le produit sur $\bfB = \bfA[\uX]/\langle\uP\rangle$:
$$
\bfB_{d'} \times \bfB_d \to \bfB_\delta, \qquad\qquad
\bfB_{d'} \times  \vH^0_\uX(\bfB)_d  \to  \vH^0_\uX(\bfB)_\delta
$$
fournit banalement des applications $\bfA$-linéaires:
$$
\bfB_{d} \to \Hom_\bfA(\bfB_{d'}, \bfB_\delta), \qquad\quad
\vH^0_\uX(\bfB)_d  \to \Hom_\bfA\big(\bfB_{d'},  \vH^0_\uX(\bfB)_\delta\big)
\leqno (\blacklozenge_1)
$$
Ci-dessous, nous supposons \fbox{$\uP$ régulière}. Cette hypothèse a
comme conséquence plusieurs propriétés fondamentales que nous passons
en revue.  Les deux premières, déjà répertoriées et utilisées, sont
conséquences du théorème de Wiebe, cf le chapitre~\ref{ObjetsSuiteP}.

\begin {enumerate}
\item
Le $\bfA$-module $\vH^0_\uX(\bfB)_\delta$ est libre de rang 1, avec le
déterminant bezoutient $\nabla$ comme base privilégiée, d'où un
isomorphisme quasi-canonique (qui consiste à identifier $1$
et~$\nabla$):
$$
\Hom_\bfA(\bfB_{d'}, \vH^0_\uX(\bfB)_\delta) \simeq (\bfB_{d'})^\star
$$
\item
L'aspect non-dégénéré de la multiplication (cf.~\ref{SemiPairingProperty}) affirmant
que l'application linéaire de gauche dans~$(\blacklozenge_1)$ est injective. 
\emph{Il en est donc de même de celle de droite}.

\item
Dans $(\blacklozenge_1)$, l'application de droite est un \fbox{isomorphisme},
ce qui complète de manière primordiale la deuxième partie (en italique) du point précédent.
\end {enumerate}

En conséquence, nous disposons d'isomorphismes explicites:
$$
\vH^0_\uX(\bfB)_d  \simeq \Hom_\bfA(\bfB_{d'}, \vH^0_\uX(\bfB)_\delta)
= \Hom_\bfA(\bfB_{d'}, \bfA\,\nabla) \simeq (\bfB_{d'})^\star
$$
Détaillons comment on passe de gauche à droite. Soit $b \in \vH^0_\uX(\bfB)_d$;
alors la forme linéaire $\ell_b \in (\bfB_{d'})^\star$ qui lui correspond
est définie par:
$$
\ell_b :  b' \mapsto [bb']\strut_\nabla \text { (composante de $bb'$ sur $\nabla$)}
\leqno(\blacklozenge_2)
$$
C'est pour cette raison que, dans le contexte précédent $d' < \min(D)$, nous avons
écrit que la famille $(\nabla_\beta)_{|\beta| = d'}$ est la base duale de la base
monomiale de $\bfB_{d'} = \bfA[\uX]_{d'}$.

\subsubsection{Une description de la transgression du bicomplexe $\rmK^{n-\sbullet}(\protect\uP\,;\bfA[\protect\uX])
  \otimes\vC_\protect\uX^{n-\sbullet}(\bfA[\protect\uX])$}

\index{transgression}%

Nous allons tenter ici de \emph{rendre compte} d'un certain
$\bfA[\uX]$-morphisme gradué $\tau$ de degré $0$, nommé transgression, intervenant dans
l'étude d'un système $\uP$ ne vérifiant aucune hypothèse particulière.
Notre description se veut concrète, mais la tâche n'est pas si aisée
puisque nous ne fournirons aucune preuve. Notre référence principale est l'article
\cite{J4}  de Jean-Pierre Jouanolou 
(3.12 Résultant, intersections complètes et résidus de Grothendieck),
en particulier (3.12.2.10) pour $\uP$ quelconque et (3.12.4.2) pour
$\uP$ régulière. Chez cet auteur, la trangression est mise en place 
par l'intermédiaire d'un couple de suites spectrales.  De notre côté, nous aurions pu la
mettre en oeuvre via l'étude du bicomplexe suivant, de type Koszul-{\Cech}:
$$
\vcenter{\hbox{
\begin {tikzpicture} [inner sep = 1pt]
\coordinate (O) at (0,0) ;
\node at (2.4,1.7) {$C'_{p,q}$} ;
\path (2.4, 1.4) edge [->] (2.4, 0.3) ;  
\path (2.5, 0.9) [right,->] node {$\uX$-\Cech{}} ;
\path (2.4, 0.9) [left] node {$d''$} ;
\path (1.9,1.7) edge [->] (0.3,1.7) ;  
\path (1.2,1.77) [above] node {$\uP$-Koszul} ;
\path (1.2,1.70) [below] node {$d'$} ;
\path (O) edge [->] (2.5,0) ; 
\path (0.1,-0.2) edge [below, <-] node {$X_\sbullet$-complexe} (2.2,-0.2) ;
\path (O) edge [->] (0,2.5) ; 
\path (0.1, 0.2) edge [rotate = 90, <-, above, midway] node [rotate=90] {$Y_\sbullet$-complexe} (2.1, 0.2) ;
\end {tikzpicture}
}}
\qquad\qquad\qquad
C'_{p,q} = 
\rmK^{n-p}(\uP\,;\bfA[\uX]) \otimes \vC_\uX^{n-q}(\bfA[\uX])
$$
Nous avons renoncé car cela aurait nécessité l'introduction de la
cohomologie de {\Cech}, suivie d'un développement trop long.  Ce
bicomplexe présente des analogies avec celui de type Koszul-Koszul
défini par $C_{p,q} = \rmK^{n-p}(\uP;\bfA[\uX])
\otimes \rmK^{n-q}(\uX;\bfA[\uX])$, que nous avons choisi pour
prouver le théorème de Wiebe (cf. annexe~\ref{WiebeProof}).  Par
exemple, grâce au fait que la suite $\uX$ est régulière, ses colonnes
sont exactes, sauf sur l'axe des abscisses, ce qui permet de définir
un morphisme précis de l'homologie du $X_\sbullet$-complexe vers celle
du $Y_\sbullet$-complexe, cf. la première section de
l'annexe~\ref{WiebeProof}.  A partir de ce morphisme précis, on
obtient, en degré homologique~$n$, la transgression en question
$\rmH^n(X_\sbullet) \to \rmH^n(Y_\sbullet)$.
Il convient cependant de tenir compte de la structure graduée,
ce qui finit par faire apparaitre la transgression sous la
forme complexe suivante (ci-dessous, le $-\Sigma$ est un twist)
$$
\tau : \Ann(\uP;\vH^n)(-\Sigma)  \longrightarrow  \vH^0_\uX(\bfB),
\qquad \text{où} \qquad
\vH^n = \vH^n_\uX(\bfA[\uX]), \quad
\Sigma = \sum_i d_i
\leqno(\blacklozenge_3)
$$
\label{TransgressionInitialForm}
Ce n'est évidemment pas la description que nous avons qualifiée de
concrète dont nous voulons rendre compte!  Note: le $\bfA[\uX]$-module
gradué $\vH^n = \vH^n_\uX(\bfA[\uX])$ qui intervient est le dernier
module de cohomologie de {\Cech} de $\uX$ sur $\bfA[\uX]$ (les autres
sont nuls); pas besoin d'être un expert en cohomologie de {\Cech} pour
le décrire (nous en reparlerons un peu plus loin).

\index{cohomologie de \Cech}%

\medskip

Pourquoi vouloir en rendre compte?  Pour plusieurs raisons.  Tout
d'abord, la notion de trangression s'accompagne de propriétés
permettant de montrer que l'application de droite
dans~$(\blacklozenge_1)$ est un isomorphisme lorsque $\uP$ est
régulière.  D'autre part, pendant la période 2015-2016, elle a
beaucoup intrigué les auteurs par la complexité de sa mise en oeuvre
en opposition à certaines de ses manifestations tangibles.  En voici
une.  Nous insistons ici sur le fait que $\uP$ ne vérifie aucune
hypothèse particulière. Soit $\mu : \bfA[\uX]_\delta \to \bfA$ une
forme linéaire nulle sur $\uPdelta$. On dispose alors de la propriété
pas du tout évidente
$$
\mu(\nabla)F - \mu(F)\nabla \in \uPdelta \quad \forall\ F \in \bfA[\uX]_\delta
$$
En conséquence, $\mu(\nabla) \in \ElimIdeal = \vH^0_\uX(\bfB)_0$.
Ceci permet de définir une application linéaire
$$
(\bfB_\delta)^\star \longrightarrow \vH^0_\uX(\bfB)_0,
\qquad\qquad
\mu \mapsto \mu(\nabla)
$$
Il s'agit de $\tau_0$, composante homogène de $\tau$ de degré $0$!
Et la propriété pas du tout évidente résulte de l'existence de $\tau$
et de ses propriétés.

\label{NOTA16-tau0}%

\bigskip

Bien qu'elle ne soit pas fondée sous la forme qui suit, nous pensons
qu'une manière assez simple de présenter la transgression $\tau$
consiste à la relier à la collection des $\bfA$-modules
$(\bfB_d)^\star$.  Une fois la mise en place suffisamment avancée,
$\tau$ peut être vue comme une application $\bfA[\uX]$-linéaire
graduée de degré $0$:
$$
\tau : \bfB^\vee(-\delta)  \longrightarrow  \vH^0_\uX(\bfB)
\overset{\rm def}{=} \uPsat/\langle\uP\rangle
$$
Nous allons décrire ce qu'est le $\bfB$-module $\bfB^\vee$
dit dual gradué de $\bfB$.  Quant au $(-\delta)$, il s'agit
d'un twist requis de manière à ce que $\tau$ soit de degré $0$. In fine,
nous obtiendrons la composante homogène de degré~$d$
de $\tau$ sous la forme suivante:
$$
\boxed{\tau_d : (\bfB_{\delta-d})^\star \longrightarrow  \vH^0_\uX(\bfB)_d}
$$
Il est important de noter d'une part la concordance entre l'indice $d$ de $\tau_d$ et
le degré de la composante homogène à \emph{l'arrivée} et d'autre part l'apparition
de $\delta-d$ dans le module de \emph{départ}. On peut considérer que la transgression
est la véritable explication structurelle du phénomène de couplage $(d, \delta-d)$.

\medskip

Afin de décrire $\bfB^\vee$, nous considérons de manière plus générale
un $\bfA[\uX]$-module~$E$.  Nous notons $E^\star_\bfA
= \Hom_\bfA(E,\bfA)$ l'ensemble des formes $\bfA$-linéaires définies
sur le $\bfA$-module $E$. On peut en faire un $\bfA[\uX]$-module en
définissant, pour $\mu \in E_\bfA^\star$ et $F \in \bfA[\uX]$, la
forme linéaire $F \cdot \mu$ par $x \mapsto \mu(Fx)$.  Lorsque $E$ est
un $\bfA[\uX]$-module gradué (c'est le cas de $E = \bfB$), on note
$E^\vee$ le sous-$\bfA[\uX]$-module de $E_\bfA^\star$ constitué des
formes $\bfA$-linéaires nulles sur les $E_d$ sauf pour un nombre fini
de $d$. On le gradue par:
$$
E^\vee_d = \{
\text{formes $\bfA$-linéaires $E \to \bfA$ nulles sur chaque $E_i$ avec $i \ne -d$}\}
$$
Ainsi \boxed{E^\vee_d = (E_{-d})^\star}.
Le coup du $d$ versus $-d$ s'explique par le fait que l'on veut 
$\bfA[\uX]_f \cdot E^\vee_d \subset E^\vee_{f+d}$:
$$
\vcenter {
\xymatrix @R = 0.5cm{
E_{-d} \ar[r]^\mu                   &\bfA  \\
E_{-f-d}\ar[u]^{\times F}\ar[ur]_{F\cdot\mu} \\
}}
\qquad\qquad
\begin {array}{l}
F \in \bfA[\uX]_f,\quad \mu \in E^\vee_d \\[3mm]
F\cdot\mu \in E^\vee_{f+d}
\end {array}
$$
\label{NOTA16-Evee}%
\index{dual gradué d'un $\bfA[\uX]$-module gradué}%

\smallskip
Les égalités $\big(\bfB^\vee(-\delta)\big)_d = (\bfB^\vee)_{-\delta+d} = (\bfB_{\delta-d})^\star$
sont en accord avec l'encadré relatif à $\tau_d$.

\bigskip

Voici les 4 propriétés importantes de $\tau$.

\begin {enumerate} [{P}-1]
\item
L'application $\tau :\bfB^\vee(-\delta)  \longrightarrow  \vH^0_\uX(\bfB)$
est une application $\bfB$-linéaire, graduée de degré $0$.
\item
$\tau_\delta : (\bfB_0)^\star \to \vH^0_\uX(\bfB)_\delta$ est caractérisée par
$\tau_\delta(\nu) = \nu(1)\overline\nabla$.
\item
$\tau_0 : (\bfB_{\delta})^\star \to \vH^0_\uX(\bfB)_0$ est caractérisée par $\tau_0(\mu)=\mu(\nabla)$.
\item
Lorsque $\uP$ est régulière, $\tau$ est un isomorphisme.
\end {enumerate}

\medskip

En conséquence de la propriété P-1, pour une forme linéaire $\mu\in(\bfB_{d'})^\star$
et $b_1 \in \bfB$, on a l'égalité $b_1\tau(\mu) = \tau(b_1\cdot\mu)$. Si on veut tenir
compte des degrés en supposant $b_1 \in \bfB_{d_1}$, cette égalité s'écrit:
$$
b_1\tau_d(\mu) = \tau_{d_1+d}(b_1\cdot\mu) \qquad \qquad
\text{avec $d = \delta-d'$}
$$
En particulier, pour $d_1 = d'$, en notant $b' \in \bfB_{d'}$ au lieu de $b_1$
$$
b'\tau_d(\mu) = \tau_{\delta}(b'\cdot\mu) \qquad \qquad
\text{avec $d = \delta-d'$}
$$
En utilisant la propriété P-2 concernant $\tau_\delta$, on obtient dans
$\vH^0_\uX(\bfB)_\delta \subset \bfB_\delta$, l'égalité
$$
\boxed{b'\tau_d(\mu) = \mu(b')\overline \nabla
\qquad
\forall\,\mu\in(\bfB_{d'})^\star,\ \forall\,b'\in\bfB_{d'}}
$$
On en tire le diagramme commutatif:
$$
\vcenter{
\xymatrix @M=0.4pc{
& (\bfB_{d'})^\star \ar[ld]_{\tau_d} \ar[dr]^-{\times \overline \nabla} &  \\
\vH^0_\uX(\bfB)_d \ar[rr]^-{b \mapsto \sbullet\times b} &
  &\Hom_{\bfA}\big(\bfB_{d'}, \vH^0_\uX(\bfB)_\delta\big)
}}
\hspace{1cm}
\text{avec \ } 
\times \overline \nabla : 
\ell \mapsto \{b' \mapsto \ell(b') \overline \nabla\}
$$
Lorsque $\uP$ est régulière, $\tau_d$ est un isomorphisme de sorte que
les 3 flèches dans le diagramme sont des isomorphismes. Ainsi
$\tau_d$ apparaît comme l'inverse de $b \mapsto \ell_b$ où $\ell_b$
est la forme linéaire définie en~$(\blacklozenge_2)$.

\bigskip

Enfin, dans l'égalité encadrée prenons le cas particulier $(d'=\delta,
d=0)$ et utilisons la propriété P-3 concernant $\tau_0$. On obtient
alors l'égalité dans $\vH^0_\uX(\bfB)_\delta \subset \bfB_\delta$ pour
toute forme linéaire $\mu\in(\bfB_\delta)^\star$ et tout
$b' \in \bfB_\delta$:
$$
b'\mu(\overline\nabla) = \mu(b')\overline\nabla
$$

\subsubsection*{Structure de $\vH^n$ et de l'isomorphisme
$\Ann(\uP;\vH^n)(-\Sigma) \simeq \bfB^\vee(-\delta)$ où $\vH^n = \vH^n_\uX(\bfA[\uX])$}

\label {NOTA16-HnCechuX}%

Comme nous l'avons déjà signalé, nul besoin d'être spécialiste de la
cohomologie de {\Cech} pour comprendre pourquoi les
$\bfA[\uX]$-modules gradués $\vH^n$ et $\bfA[\uX]^\vee(n)$ sont
canoniquement isomorphes où nous notons $\vH^n = \vH^n_\uX(\bfA[\uX])$
pour alléger. Nous suivons d'assez près la thèse de C. Tête en \cite[section II.4]{Tete}.
Nous avons juste besoin de la dernière différentielle du
complexe de {\Cech} de $\uX$, qui s'exprime de la manière suivante
entre localisés monomiaux de $\bfA[\uX]$:
$$
\bigoplus_{i=1}^n \bfA[\uX]_{X_1\cdots X_{i-1} \widehat{X_i} X_{i+1} \cdots X_n}
\longrightarrow  \bfA[\uX]_{X_1\cdots X_n}
\qquad \qquad
\bigoplus_{i=1}^n F_i \longmapsto
\sum_{i=1}^n (-1)^{i-1} (F_i)_{X_1\cdots X_n}
$$
Il s'ensuit que $\vH^n$ est le $\bfA[\uX]$-module gradué quotient:
$$
\vH^n = \dfrac{\bfA[\uX]_{X_1\cdots X_n}} {
\bfA[\uX]_{\widehat{X_1}X_2\cdots X_n} + \bfA[\uX]_{X_1\widehat{X_2}\cdots X_n} + \cdots +
\bfA[\uX]_{X_1X_2\cdots \widehat{X_n}}
}
$$
Pour cerner ce $\bfA[\uX]$-module, regardons-le en tant que
$\bfA$-module.  Le numérateur, qui n'est autre que $\bfA[X_1^{\pm 1},\dots,
X_n^{\pm1}]$, admet pour $\bfA$-base les monômes $X^\beta$ avec
$\beta\in\bbZ^n$ tandis que le dénominateur admet pour $\bfA$-base les
$X^\beta$ avec $\beta\in\bbZ^n$ possédant un indice $i$ tel que
$\beta_i \ge 0$. Le quotient admet donc comme $\bfA$-base les
$X^\beta$ vérifiant pour tout $i$ l'inégalité $\beta_i < 0$ ou encore
$\beta_i \le -1$. Ce que nous écrirons $\beta \preceq -\Un$ en
désignant par $\preceq$ la relation d'ordre sur $\bbZ^n$ composante à
composante.  Ces monômes $X^\beta$ peuvent s'écrire $(X_1\cdots
X_n)^{-1} X^\gamma$ avec $\gamma \preceq 0$ si bien que:
$$
\vH^n = (X_1\cdots X_n)^{-1} \bfA[X_1^{-1}, \dots, X_n^{-1}]
$$
Sa structure de $\bfA[\uX]$-module est donnée, pour $\alpha \succeq 0$
et $\beta\preceq -\Un$, par:
$$
X^\alpha \cdot X^\beta = \begin {cases}
X^{\alpha+\beta} & \text{si $\alpha+\beta \preceq -\Un$} \\
0&\text{sinon} \\
\end {cases}
$$
En particulier pour les composantes homogènes de $\vH^n$:
$$
\vH^n_{-n} = \bfA (X_1\cdots X_n)^{-1} \simeq \bfA, \qquad\qquad
\vH^n_d = 0 \quad \forall\, d > -n
$$
La multiplication de $\bfA[\uX]$ sur $\vH^n$ induit une forme bilinéaire 
$$
\Theta : \bfA[\uX]_{-n-d} \times \vH^n_d \longrightarrow \vH^n_{-n} \simeq \bfA
$$
Pour $X^\alpha\in\bfA[\uX]_{-n-d}$ et $X^\beta\in\vH^n_d$, il n'est
pas difficile de voir que $\Theta(X^\alpha,X^\beta) = 1$ si
$\alpha+\beta = -\Un$ et $0$ sinon. Cette forme bilinéaire établit donc
une dualité parfaite d'où un isomorphisme canonique 
$\vH^n_d \overset{\rm can.}{\simeq} (\bfA[\uX]_{-n-d})^\star$ défini par 
$h \mapsto \Theta(\sbullet,h)$. Pour rallier $\bfA[\uX]^\vee$
et son twist $\bfA[\uX]^\vee(n)$, on écrit:
$$
\vH^n_d \simeq (\bfA[\uX]^\vee)_{n+d} = \big(\bfA[\uX]^\vee(n)\big)_{d}
$$
En définitive $\vH^n \simeq \bfA[\uX]^\vee(n)$ via $h \mapsto \Theta(\sbullet,h)$.
On laisse le soin au lecteur d'en déduire
$$
\Ann(\uP;\vH^n)  \simeq (\bfA[\uX]/\langle\uP\rangle)^\vee(n) \overset{\rm def.}{=} \bfB^\vee(n)
$$
Il faut maintenant tenir compte du twist $-\Sigma$  intervenant
en $(\blacklozenge_3)$ page~\pageref{TransgressionInitialForm}
où $\Sigma = \sum_i d_i$. Puisque $-\Sigma + n = -\delta$,
il vient:
$$
\Ann(\uP;\vH^n)(-\Sigma) \simeq \bfB^\vee(-\delta)
$$
Ainsi s'achève notre objectif principal consistant à remplacer la
version de $\tau$ à gauche par celle de droite:
$$
\tau : \Ann(\uP;\vH^n)(-\Sigma)  \longrightarrow  \vH^0_\uX(\bfB)
\qquad\text{versus}\qquad
\tau : \bfB^\vee(-\delta)  \longrightarrow  \vH^0_\uX(\bfB)
$$

\subsection{Le cas particulier $n=3$ et $d_1=d_2=d_3$: le format $D = (e,e,e)$}

On rappelle qu'en 3 variables $X,Y,Z$:
$$
\dim \bfA[\uX]_m = \binom{2 + m}{2} = \dfrac{(m+2)(m+1)}{2}
$$
Le degré critique de $D = (e,e,e)$ est $\delta = 3(e-1)$. Comme $\min(D) = e$, la
fourchette $[\delta-\min(D)+1 .. \delta]$ des degrés $d$ est constituée
des entiers $2e-2, 2e-1, 2e, \dots, 3(e-1)$. Dans cette fourchette,
il y a deux entiers $d$ qui vérifient $\Jex_{2,d} = 0$ \idest{} $d
< \min_{i \ne j} (d_i+d_j) = 2e$ et qui fournissent donc pour le résultant
une formule déterminantale d'ordre $\dim\bfA[\uX]_d$ sans dénominateur ; il s'agit  de :
$$\
\boxed{d=2e-2},\ \dim\bfA[\uX]_d = (2e-1)e, \quad \text{ et } \quad 
d=2e-1,\ \dim\bfA[\uX]_d = (2e+1)e
$$
Après Sylvester (cf \cite[point 91]{Salmon}),  nous allons étudier le cas encadré qui
conduit à un déterminant d'ordre $(2e-1)e$ pour le résultant.

\medskip

L'entier complémentaire à $\delta$ de $d$ est \boxed{d' = e-1}. Pour chaque
exposant $\beta = (i,j,k)$ tel que $|\beta| = d'$ \idest{} $i+j+k = e-1$,
il existe une matrice homogène, non unique, que l'on note $\dsV_\beta$:
$$
\uP = [X^{i+1}, Y^{j+1}, Z^{k+1}]\, \dsV_\beta \qquad
\dsV_\beta = 
\begin {bmatrix}
V_{11} & V_{12} & V_{13} \\
V_{21} & V_{22} & V_{23} \\
V_{31} & V_{32} & V_{33} \\
\end {bmatrix}
$$
Son déterminant $\nabla_\beta = \det \dsV_\beta$ est homogène de degré $d$,
ce $d$ provenant de $\sum_\ell d_\ell - (|\beta| + n) = \delta -d'$.

\medskip

La matrice strictement carrée $\bsOmega_{d,\uP}$ indexée par les
monômes de $\bfA[\uX]_d$ est constituée de
$\dim \Jex_{1,d}$ colonnes (elles sont multiples des $P_i$, de degré $d$)
et de $\dim \Smac_{0,d} = \dim\bfA[\uX]_{d'}$ colonnes-argument
destinées à être évaluées en la famille
$(\nabla_\beta)_{|\beta|=d'}$. On a:
$$
\Jex_{1,d} = X^e\bfA[\uX]_{d-e} \oplus Y^e\bfA[\uX]_{d-e} \oplus Z^e\bfA[\uX]_{d-e} 
$$
D'où les dimensions:
$$
\dim \Jex_{1,d} = 3\dfrac{(d-e+2)(d-e+1)}{2} = 3\dfrac{e(e-1)}{2},
\qquad
\dim \Smac_{0,d} = \dfrac{(d'+2)(d'+1)}{2} = \dfrac{(e+1)e}{2} 
$$
Tout est prêt pour élaborer la matrice strictement carrée sur $\bfA[\uX]_d$
dont le déterminant est le résultant des 3 polynômes homogènes $(P_1, P_2, P_3)$ de degré $e$.
Sa spécialisation en le jeu étalon $(X^e, Y^e, Z^e)$ est $\Id_{\bfA[\uX]_d}$ à condition
d'avoir choisi $\dsV_\beta$ de manière à ce que sa spécialisation soit
$\diag(\uX^{\emouton-\beta})$.
Ci-après, on a visualisé le cas cubique $e=3$.

\bigskip

On peut préférer à cette étude algébrique une approche plus géométrique,
comme par exemple dans l'ouvrage \cite[exercice~15, chap.~3,
  \S4]{Cox}.  Voici un rapide résumé de la stratégie.  Les auteurs, en
suivant probablement \cite[point 91]{Salmon}, introduisent la famille
suivante de polynômes homogènes de degré $d$ où les~$m_i$ sont des
monômes:
$$
(m_1P_1)_{\deg m_1=d-e},\quad (m_2P_2)_{\deg m_2=d-e},\quad (m_3P_3)_{\deg m_3=d-e}, \quad
(\nabla_\beta)_{|\beta|=d'}
$$
On n'oublie pas que $d-e = e-2$ et $d' = e-1$. On vérifie que cette
famille est de cardinal $(2e-1)e$ qui est aussi le nombre de monômes
en 3 variables de degré $d = 2e-2$.  On peut alors former une matrice
carrée~$C$, de taille $(2e-1)e$, exprimant cette famille dans la base
monomiale de $\bfA[\uX]_d$.  Comme cette famille a les mêmes zéros
projectifs que $\uP = (P_1,P_2,P_3)$, $\det(C)$ est multiple de $\Res(\uP)$.
Il faut ensuite contrôler le fait que $\det(C)$ est homogène en chaque $P_i$,
de poids~$e^2$ (notre $\widehat d_i$). Ce $e^2$ provient de
$$
\dfrac{e(e-1)}{2}  + \dfrac{(e+1)e}{2}  
$$
A gauche, c'est le nombre de colonnes de type $P_i$ i.e. $\dim \Jex_{1,d}^{(i)} = \dfrac{1}{3}
\dim \Jex_{1,d}$ et à droite, c'est le nombre de colonnes-argument i.e. $\dim \Smac_{0,d}$
(chaque $\nabla_\beta$ est homogène de poids 1 en $P_i$).
Ainsi $\det(C) = \lambda\, \Res(\uP)$
où $\lambda$ est dans le corps de base. La spécialisation de $\det(C)$
en $(X^e, Y^e, Z^e)$ est $\pm 1$. Bilan: $\Res(\uP) = \pm \det(C)$.

\subsubsection*{L'exemple du format $D = (3,3,3)$ de degré critique $\delta = 6$, $d=4$, $d' := \delta-d=2$}

Voici les dimensions qui interviennent
$$
\dim \Smac_{1,d} = \dim \Jex_{1,d} = 9, \qquad \dim \Smac_{0,d} = 6, \qquad
\dim \rmK_{0,d} = 9 + 6 = 15
$$
Le tableau ci-dessous est celui de l'isomorphisme mouton-swap
$: \Smac_{0,d} \simeq \bfA[\uX]_{d'}$. La ligne supérieure est la base
monomiale de $\Smac_{0,d}$ et le produit d'un monôme par le monôme en
dessous est le mouton-noir $X^2Y^2Z^2$.
Dans l'autre sens: la ligne inférieure est la base monomiale
de $\bfA[\uX]_{d'}$,~etc.
$$
\begin {array} {c|c|c|c|c|c|c}
X^\alpha        &X^2Y^2 &X^2YZ  &X^2Z^2 &XY^2Z  &XYZ^2 &Y^2Z^2\\ [2mm]
\hline
\vrule height12pt depth3pt width0pt
X^{D-\Un-\alpha} &Z^2   &YZ      &Y^2    &XZ     &XY    &X^2 \\
\end {array}
$$

La numérotation des coefficients des 3 polynômes cubiques est la suivante
$$
\setlength{\tabcolsep}{2pt}
\left\{
\begin{tabular}{rcp{14cm}} 
$P_{1}$ & $=$ & $a_{1}X^{3} + a_{2}X^{2}Y + a_{3}X^{2}Z + a_{4}XY^{2} + a_{5}XYZ + a_{6}XZ^{2} + a_{7}Y^{3} + a_{8}Y^{2}Z + a_{9}YZ^{2} + a_{10}Z^{3}$\\ [0.1cm] 
$P_{2}$ & $=$ & $b_{1}X^{3} + b_{2}X^{2}Y + b_{3}X^{2}Z + b_{4}XY^{2} + b_{5}XYZ + b_{6}XZ^{2} + b_{7}Y^{3} + b_{8}Y^{2}Z + b_{9}YZ^{2} + b_{10}Z^{3}$\\ [0.1cm] 
$P_{3}$ & $=$ & $c_{1}X^{3} + c_{2}X^{2}Y + c_{3}X^{2}Z + c_{4}XY^{2} + c_{5}XYZ + c_{6}XZ^{2} + c_{7}Y^{3} + c_{8}Y^{2}Z + c_{9}YZ^{2} + c_{10}Z^{3}$\\ [0.1cm] 
\end{tabular}
\right.
$$
Ce qui conduit à la matrice $\bsOmega_{d,\uP}$ ci-dessous:
$$ \def \lra{\leftrightarrow}
\bsOmega_{d,\uP} = 
\NorthEastBordermatrix{
\Veti{X^4 \lra X\,e_{1}} &\Veti{X^3Y \lra Y\,e_{1}} &\Veti{X^3Z \lra Z\,e_{1}} &\Veti{XY^3 \lra X\,e_{2}} &\Veti{XZ^3 \lra X\,e_{3}} &\Veti{Y^4 \lra Y\,e_{2}} &\Veti{Y^3Z \lra Z\,e_{2}} &\Veti{YZ^3 \lra Y\,e_{3}} &\Veti{Z^4 \lra Z\,e_{3}} &
\Veti{X^2Y^2\leftrightarrow Z^{2}} &\Veti{X^2YZ\leftrightarrow YZ} &\Veti{X^2Z^2\leftrightarrow Y^{2}} &
   \Veti{XY^2Z\leftrightarrow XZ} &\Veti{XYZ^2\leftrightarrow XY} &\Veti{Y^2Z^2\leftrightarrow X^{2}} & \\
a_{1} & . & . & b_{1} & c_{1} & . & . & . & . &\sbullet &\sbullet &\sbullet &\sbullet &\sbullet &\sbullet & \Heti{X^{4}} \\
a_{2} & a_{1} & . & b_{2} & c_{2} & b_{1} & . & c_{1} & . &   &   &   &   &   &   & \Heti{X^{3}Y} \\
a_{3} & . & a_{1} & b_{3} & c_{3} & . & b_{1} & . & c_{1} &   &   &   &   &   &   & \Heti{X^{3}Z} \\
a_{7} & a_{4} & . & b_{7} & c_{7} & b_{4} & . & c_{4} & . &   &   &   &   &   &   & \Heti{XY^{3}} \\
a_{10} & . & a_{6} & b_{10} & c_{10} & . & b_{6} & . & c_{6} &   &   &   &   &   &   & \Heti{XZ^{3}} \\
. & a_{7} & . & . & . & b_{7} & . & c_{7} & . &   &   &   &   &   &   & \Heti{Y^{4}} \\
. & a_{8} & a_{7} & . & . & b_{8} & b_{7} & c_{8} & c_{7} &   &   &   &   &   &   & \Heti{Y^{3}Z} \\
. & a_{10} & a_{9} & . & . & b_{10} & b_{9} & c_{10} & c_{9} &   &   &   &   &   &   & \Heti{YZ^{3}} \\
. & . & a_{10} & . & . & . & b_{10} & . & c_{10} &   &   &   &   &   &   & \Heti{Z^{4}} \\
a_{4} & a_{2} & . & b_{4} & c_{4} & b_{2} & . & c_{2} & . &   &   &   &   &   &   & \Heti{X^{2}Y^{2}} \\
a_{5} & a_{3} & a_{2} & b_{5} & c_{5} & b_{3} & b_{2} & c_{3} & c_{2} &   &   &   &   &   &   & \Heti{X^{2}YZ} \\
a_{6} & . & a_{3} & b_{6} & c_{6} & . & b_{3} & . & c_{3} &   &   &   &   &   &   & \Heti{X^{2}Z^{2}} \\
a_{8} & a_{5} & a_{4} & b_{8} & c_{8} & b_{5} & b_{4} & c_{5} & c_{4} &   &   &   &   &   &   & \Heti{XY^{2}Z} \\
a_{9} & a_{6} & a_{5} & b_{9} & c_{9} & b_{6} & b_{5} & c_{6} & c_{5} &   &   &   &   &   &   & \Heti{XYZ^{2}} \\
. & a_{9} & a_{8} & . & . & b_{9} & b_{8} & c_{9} & c_{8} &\sbullet &\sbullet &\sbullet &\sbullet &\sbullet &\sbullet & \Heti{Y^{2}Z^{2}} \\
}
$$ 
Il ne reste plus qu'à reporter dans les 6 colonnes de droite les
coefficients des déterminants $(\nabla_\beta)_{|\beta|=d'}$ (ce sont
des polynômes homogènes de degré $d=4$), en respectant l'ordre bien
entendu. Le déterminant obtenu est le résultant de $\uP = (P_1,P_2,P_3)$.

\subsubsection*{Le cas particulier $D = (2,2,2)$ de degré critique $\delta=3$, $d=2$, $d'=1$
                 : la formule de Sylvester}

La numérotation des coefficients de $P_1$ est la suivante (idem pour $P_2, P_3$ avec les lettres $b,c$):
$$
P_1 = a_1X^2 + a_2XY + a_3XZ + a_4Y^2 + a_5YZ + a_6Z^2
$$
Déterminons une famille $(\nabla_\beta)_{|\beta|=d'}$. Pour cela, il nous faut définir (de manière non unique) 3 matrices homogènes $\dsV_X, \dsV_Y, \dsV_Z$ 
vérifiant:
$$
\uP = [X^2,Y,Z]\,\dsV_X = [X,Y^2,Z]\,\dsV_Y = [X,Y,Z^2]\,\dsV_Z
$$
Par exemple (on montre uniquement la première colonne, les autres sont analogues):
$$
\text{col}_1(\dsV_X) =
\begin {bmatrix}
a_1 \\
a_2X + a_4Y + a_5Z \\
a_3X + a_6Z \\
\end {bmatrix}
\quad
\text{col}_1(\dsV_Y) =
\begin {bmatrix}
a_1X + a_2Y \\
a_4 \\
a_3X + a_5Y + a_6Z \\
\end {bmatrix}
\quad
\text{col}_1(\dsV_Z) =
\begin {bmatrix}
a_1X + a_2Y + a_3Z\\
a_4Y + a_5Z \\
a_6 \\
\end {bmatrix}
$$
Notons $\nabla_\sbullet=\det\dsV_\sbullet$ (polynôme homogène de degré 2).
Utilisons la base monomiale visible sur $P_1$ ; ci-dessous, les monômes de $\Jex_{1,d}$ sont soulignés 
et les monômes de $\Smac_{0,d}$ sont transformés par le mouton-swap :
$$
\newcommand \vetibis[1] {\rotatebox{-90}{\mbox{\!\!\!\!$\scriptstyle#1$}}}
\begin{array}{cccccc}
\underline{X^2} & XY & XZ & \underline{Y^2} & YZ & \underline{Z^2} \\
& \vetibis{\leftrightsquigarrow} & \vetibis{\leftrightsquigarrow} & & \vetibis{\leftrightsquigarrow} & \\[0.1cm]
 & Z  & Y  &  & X  & 
\end{array}
$$
Le résultant s'obtient comme un déterminant d'ordre 6:
$$
\Res(\uP) \ =\  \big[P_1, \nabla_Z, \nabla_Y, P_2, \nabla_X, P_3\big]
\qquad\qquad
\leqno (\heartsuit)
$$
{\bf Lien avec le jacobien $J$ de $\uP$}.
Les $\nabla_\sbullet$ sont reliés modulo $\langle\uP\rangle$ aux dérivées partielles du jacobien:
$$
2^3\nabla_X \equiv J'_X,\qquad 2^3\nabla_Y \equiv J'_Y,\qquad 2^3\nabla_Z \equiv J'_Z
$$
Donc, lorsque cela a du sens de diviser par $2$, on peut aménager $(\heartsuit)$ pour tomber
sur la formule de Sylvester:
$$
\Res(\uP) \ =\  \frac{1}{2^9} \big[P_1, J'_Z, J'_Y,  P_2, J'_X, P_3\big]
$$
Les congruences modulo $\uP$ ci-dessus sont conséquences des identités suivantes:
$$
2^3\nabla_X - J'_X = 2[23\uP] + 4[15\uP], \qquad
2^3\nabla_Y - J'_Y = 2[52\uP] + 4[43\uP],\qquad
2^3\nabla_Z - J'_Z = 2[35\uP] + 4[62\uP]
$$
où l'on a posé:
$$
[ij\uP] = \begin {vmatrix}
a_i & a_j & P_1 \\
b_i & b_j & P_2 \\
c_i & c_j & P_3 \\
\end {vmatrix}
\qquad
\text{qui appartient à $\bfA P_1 + \bfA P_2 + \bfA P_3$}
$$
{\bf Variante.}
En appliquant la formule d'Euler une première fois à $P_i$ homogène de degré 2:
$$
2P_i = X(\partial P_i/\partial X) + Y(\partial P_i/\partial Y) + Z(\partial P_i/\partial Z)
$$
puis une seconde fois à $\partial P_i/\partial X$ de degré 1, on obtient
$$
2P_i = X^2 \dfrac{\partial^2 P_i}{\partial^2 X} +
Y \left(X\dfrac{\partial^2 P_i}{\partial X\partial Y} + \dfrac{\partial P_i}{\partial Y}\right) +
Z \left(X\dfrac{\partial^2 P_i}{\partial X\partial Z} + \dfrac{\partial P_i}{\partial Z}\right) 
$$
Ceci permet d'expliciter une matrice homogène $\dsV'_X$ telle que $2\,\uP = [X^2,Y,Z]\, \dsV'_X$
(attention à ne pas oublier le facteur 2).  En voici la première colonne:
$$
\text{col}_1(\dsV'_X) = \begin {bmatrix}
2a_1 \\
2a_2X + 2a_4Y + a_5Z \\
2a_3X + a_5Y + 2a_6Z \\
\end {bmatrix}
$$
On a bien sûr deux autres matrices analogues $\dsV'_Y$, $\dsV'_Z$ pour
$[X,Y^2,Z]$ et $[X,Y,Z^2]$.
$$
\text{col}_1(\dsV'_Y) = \begin {bmatrix}
2a_1X + 2a_2Y + a_3Z \\
2a_4 \\
a_3X + 2a_5Y + 2a_6Z \\
\end {bmatrix},
\qquad\qquad
\text{col}_1(\dsV'_Z) = \begin {bmatrix}
2a_1X + a_2Y + 2a_3Z \\
a_2X + 2a_4Y + 2a_5Z \\
2a_6 \\
\end {bmatrix}
$$
Les déterminants de ces 3 matrices sont égaux modulo $\uP$ aux dérivées partielles
du jacobien $J$ de $\uP$, car :
$$
\det \dsV'_X = J'_X  + 2[23\uP], \qquad
\det \dsV'_Y = J'_Y  + 2[52\uP], \qquad
\det \dsV'_Z = J'_Z  + 2[35\uP]
$$
Pour essayer de s'y retrouver dans ces identités, voici une petite correspondance : 
pour la première identité liée à $X$, faire le rapprochement $2[23\uP] \leftrightarrow (2a_2X,
2a_3X)$ que l'on voit dans la colonne de $\dsV'_X$.  Idem pour la
seconde liée à $Y$: $2[52\uP] \leftrightarrow (2a_5Y, 2a_2Y)$ et la
dernière à $Z$ : $2[35\uP] \leftrightarrow (2a_3Z,2a_5Z)$.
Bref, en tenant compte du facteur 2
dans les matrices (ce qui introduit $2^3$ dans les déterminants), et
en prenant les mêmes précautions que pour~$(\heartsuit)$, voici de nouveau
la formule de Sylvester utilisant le jacobien:
$$
\Res(\uP) \ = \ \frac{1}{2^9} \big[P_1, \det\dsV'_Z, \det\dsV'_Y,  P_2,\det\dsV'_X, P_3\big] \ =\ 
\frac{1}{2^9} \big[P_1, J'_Z, J'_Y,  P_2, J'_X, P_3\big]
$$

\cleardoublepage


\section {Le bezoutien et la méthode de Morley}
\label{SectionBezoutMorley}

Comme dans presque chaque chapitre, $D = (d_1, \ldots, d_n)$ désigne
un format fixé de degrés que, volontairement, nous ne faisons pas
figurer dans les notations, comme par exemple dans le degré critique
$\delta = \sum_i(d_i-1)$ ou bien la caractéritique d'Euler-Poincaré du
complexe $\rmK_{\sbullet,d}(\uX^D)$ notée $\chi_d$:
$$
\chi_d = \dim\bfA[\uX]_d/\Jex_{1,d} \overset{\rm def.}{=}
\dim \bfA[\uX]_d/\langle X_1^{d_1}, \ldots, X_n^{d_n}\rangle_d =
\#\{\alpha \mid \alpha \preccurlyeq \emouton  \text{ et } |\alpha| = d \}
$$
Comme d'habitude, nous allons construire, à partir d'un système $\uP =
(P_1, \ldots, P_n)$ de polynômes homogènes de format de degrés $D$, un
certain nombre d'objets \og relevant de la théorie de l'élimination\fg.
Ce qui demande parfois un peu de réflexion, c'est la dépendance
en $D$ ou en $\uP$, mais nous ferons en sorte de préciser.

\subsubsection*{Préambule sur la bihomogénéité dans $\bfA[\uX,\uY]$}

La nouveauté ici, par rapport aux
chapitres antérieurs, c'est l'utilisation d'un second jeu
$\uY = (Y_1, \ldots, Y_n)$ d'indéterminées, de  l'algébre $\bfA[\uX,\uY]$
ainsi que des idéaux $\langle \uP(\uX) - \uP(\uY)\rangle$ et
$\langle \uP(\uX), \uP(\uY)\rangle$. Une notion importante est
la bihomogénéité. Un polynôme de $\bfA[\uX,\uY]$ est bihomogène
de bidegré $(p,q)$ s'il est combinaison $\bfA$-linéaire de monômes
$X^\alpha Y^\beta$ avec $|\alpha| = p$, $|\beta| = q$, auquel cas
il est homogène de degré~$p+q$. On note $\bfA[\uX,\uY]_{p,q}$
le $\bfA$-module composante bihomogène de bidegré $(p,q)$:
$$
\bfA[\uX,\uY]_{p,q} \ = \ \bigoplus_{|\alpha|=p \atop |\beta|=q} \bfA X^\alpha Y^\beta
$$
On a donc
$$
\bfA[\uX,\uY] = \bigoplus_{r \in \bbN} \bfA[\uX,\uY]_r, \qquad
\bfA[\uX,\uY]_r = \bigoplus_{p+q=r} \bfA[\uX,\uY]_{p,q}
$$
Les polynômes $P_i(\uX), P_i(\uY)$ sont bihomogènes, le premier de bidegré $(d_i,0)$,
le second de bidegré~$(0,d_i)$. En conséquence,
l'idéal $\langle\uPX,\uPY \rangle$ de $\bfA[\uX,\uY]$ est bihomogène au sens
où pour tout $F$ de cet idéal, chaque composante
bihomogène $F_{p,q}$ de bidegré $(p,q)$ de $F$ est encore dans l'idéal.
Un polynôme bihomogène de $\bfA[\uX,\uY]$ appartient au $\bfA$-module 
$\langle \uPX, \uPY \rangle_{p,q}$ s'il est combinaison-$\bfA$-linéaire 
de polynômes de deux types 
$$
\left\{
\begin{tabular}{ll}
$H_{\beta}(\uX)\, Y^\beta$ & 
avec $H_\beta(\uX) \in \langle \uPX \rangle_p$ et $|\beta | = q$ \\ [0.2cm]
$X^\alpha \, H_{\alpha}(\uY)$ & avec $H_\alpha(\uY) \in \langle \uPY \rangle_{q}$ et $|\alpha | = p$ 
\end{tabular}
\right.
$$
Pour $q < \min(D)$,  on en déduit, puisque $\langle \uPY \rangle_{q} = 0$, que:
$$
\langle \uPX,\uPY \rangle_{p,q} = \bigoplus_{|\beta|=q} \langle \uP(\uX)\rangle_p\, Y^\beta
$$
On utilisera ce résultat à plusieurs reprises.

\label{NOTA17-AXY}%
\label{NOTA17-AXYpq}%

\medskip
En revanche, le polynôme $P_i(\uX) - P_i(\uY)$ n'est \emph{pas}
bihomogène et l'idéal $\langle \uPX-\uPY \rangle$ non plus.
Cela explique parfois la \og relâche\fg{} consistant à utiliser l'inclusion 
$\langle \uPX-\uPY \rangle \subset \langle \uPX,\uPY \rangle$ de manière à en déduire
certaines appartenances de composantes bihomogènes.

\subsection{Matrices et déterminants bezoutiens
  $\Bez(\protect\uX,\protect\uY)$ d'un système $\protect\uP$}

\`A partir de maintenant, $d$ et $d'$ sont deux entiers complémentaires à $\delta$, 
c'est-à-dire $d+d' = \delta$.

\begin{defn}\label{DefBezoutien}
On appelle {\rm matrice bezoutienne} d'un système $\uP$ toute matrice homogène
$\dsV := \dsV(\uX,\uY)$ transformant la {\rm ligne} $\uX - \uY$ en la {\rm ligne} $\uPX - \uPY$ au sens suivant:
$$
\begin{bmatrix}
P_1(\uX)-P_1(\uY) & \cdots &  P_n(\uX)-P_n(\uY)
\end{bmatrix} 
\ = \ 
\begin{bmatrix}
X_1-Y_1 & \cdots &  X_n-Y_n
\end{bmatrix} \dsV
$$
homogène signifiant que $\dsV_{ij}$ est polynôme homogène de degré $d_j-1$ en les indéterminées $\uX, \uY$.

\index{matrice!bezoutienne2@$(\uX,\uY)$-bezoutienne}%
\index{matrice!bezoutienne2@$(\uX,\uY)$-bezoutienne!rigidifiée}%

Le déterminant de $\dsV$, noté $\Bez := \Bez(\uX, \uY)$, est appelé
{\rm bezoutien} de $\uP$.  C'est un polynôme homogène de degré
$\delta$ en les indéterminées~$\uX,\uY$.  On note $\Bez_{d,d'}$ sa
composante bihomogène de bidegré $(d,d')$. Comme $\Bez \in
\bfA[\uX,\uY]_\delta$, on a $\Bez = \sum_{d+d' = \delta} \Bez_{d,d'}$.

On désigne par $\Bez_\beta(\uX)$ le coefficient en~$Y^\beta$ de~$\Bez$ 
vu comme polynôme en $\uY$ de sorte que:
$$
\Bez(\uX,\uY) = \sum_{\beta} \Bez_\beta(\uX) Y^\beta
$$
On appelle {\rm matrice bezoutienne rigidifiée} la matrice $\dsV \in \bbM_n(\bfA[\uX,\uY])$ définie par 
$$
\dsV_{i,j} \ = \ 
\dfrac{P_j(Y_1, \dots, Y_{i-1}, X_i, X_{i+1},\dots,X_n) - 
P_j(Y_1,\dots,Y_{i-1}, Y_i, X_{i+1}, \dots, X_n)}{X_i-Y_i}
$$
On a ainsi $\dsV_{i,j} \in \bfA[X_i, \dots, X_n, Y_1, \dots, Y_i]$.
\end{defn}

\smallskip

\label{NOTA17-Bez}%

Chez certains auteurs, la matrice bezoutienne rigidifiée est
\emph{transposée} de la nôtre. Histoire d'enfoncer le clou, nous
montrons la nôtre pour $n=2$ et partiellement pour $n=3$, pour bien
voir que chaque $P_j$ apparaît en colonne et $X_i-Y_i$ en ligne:
$$
\begin{bmatrix}
\dfrac{P_1(X_1,X_2) - P_1(Y_1,X_2)}{X_1-Y_1} & \dfrac{P_2(X_1,X_2) - P_2(Y_1,X_2)}{X_1-Y_1} \\
\noalign {\medskip}
\dfrac{P_1(Y_1,X_2) - P_1(Y_1,Y_2)}{X_2-Y_2} & \dfrac{P_2(Y_1,X_2) - P_2(Y_1,Y_2)}{X_2-Y_2} \\
\end{bmatrix}
\qquad
\begin{bmatrix}
\dfrac{P_j(X_1,X_2,X_3) - P_j(Y_1,X_2,X_3)}{X_1-Y_1} \\
\noalign {\medskip}
\dfrac{P_j(Y_1,X_2,X_3) - P_j(Y_1,Y_2,X_3)}{X_2-Y_2} \\
\noalign {\medskip}
\dfrac{P_j(Y_1,Y_2,X_3) - P_j(Y_1,Y_2,Y_3)}{X_3-Y_3} \\
\end{bmatrix}
$$

\begin{rmqs} \leavevmode
\label{BezoutianRemarks}
  
\medskip  
$\rhd$
Afin d'éviter les confusions entre les notions \og bezoutiennes\fg{} de $\bfA[\uX]$
et celles de $\bfA[\uX,\uY]$, nous qualifierons les premières à l'aide d'un préfixe
$\uX$ comme annoncé dans la définition~\ref{DefNabla}.

Ainsi, la spécialisation $\uY := 0$ dans $\dsV$ conduit à une matrice
$\uX$-bezoutienne pour $\uP$ (une matrice transformant la \emph{ligne} $\uX$
en la \emph{ligne} $\uP$) de sorte que $\Bez(\uX,0) = \Bez_{\delta,
 0}$ est un déterminant $\uX$-bezoutien de~$\uP$. Donc 
$\Bez(\uX,0) \equiv \nabla(\uX) \bmod \uPdelta$ pour n'importe
quel déterminant $\uX$-bezoutien $\nabla(\uX)$ de~$\uP$.

\medskip
$\rhd$
Donnons des précisions sur la matrice bezoutienne rigidifiée.
Analysons le cas où $P_j$ est un monôme $X^\alpha$.
En notant $Y_{[1..i[}^\alpha = Y_1^{\alpha_1} \cdots Y_{i-1}^{\alpha_{i-1}}$ 
et $X_{]i..n]}^\alpha = X_{i+1}^{\alpha_{i+1}} \cdots X_n^{\alpha_n}$,
alors la colonne $j$ de l'égalité matricielle~\ref{DefBezoutien}
fournit l'égalité télescopique:
$$
X^\alpha - Y^\alpha = \sum_i Q_{i,\alpha} \times (X_i - Y_i)
$$
où
$$
Q_{i,\alpha} \ = \ 
\dfrac{Y_{[1..i[}^\alpha\,X_i^{\alpha_i}\,X_{]i..n]}^\alpha 
- 
Y_{[1..i[}^\alpha\,Y_i^{\alpha_i}\,X_{]i..n]}^\alpha}{X_i-Y_i}
\ = \ 
Y_{[1..i[}^\alpha \ \dfrac{X_i^{\alpha_i}-Y_i^{\alpha_i}}{X_i-Y_i} \ X_{]i..n]}^\alpha
$$
Ce qui coïncide avec la formule donnée par Jouanolou en 3.11.19.12 de~\cite{J7}.
Par combinaison linéaire, on obtient alors la formule close suivante :
$$
\dsV_{i,j} \ = \ 
\sum_{|\alpha|=d_j} p_{j,\alpha} \ 
Y_{[1..i[}^\alpha \ \dfrac{X_i^{\alpha_i}-Y_i^{\alpha_i}}{X_i-Y_i} \ X_{]i..n]}^\alpha
\qquad \text{où } \quad 
P_j = \sum_{|\alpha| = d_j} p_{j,\alpha} X^\alpha
$$
\medskip
$\rhd$
On peut tordre la définition de $Q_{i,\alpha}$ à l'aide d'une permutation
$\sigma \in \fS_n$ en remplaçant les intervalles
de la relation d'ordre $<$ habituelle qui y interviennent par ceux de
la relation d'ordre $<_\sigma$. Par exemple:
$$
[1..i[_\sigma = \{k \mid  1 \le_\sigma k <_\sigma i\} =
\big\{k \mid  \sigma^{-1}(1) \leqslant \sigma^{-1}(k) < \sigma^{-1}(i)\big\} =
\sigma\big([\sigma^{-1}(1) .. \sigma^{-1}(i)[\big)
$$
En prenant $n$ permutations $\sigma_1, \dots, \sigma_n \in\fS_n$, la permutation $\sigma_j$ étant associée
à $P_j$, on obtient d'autres matrices bezoutiennes et bezoutiens de $\uP$.

\medskip
$\rhd$
La matrice diagonale de $i$-ème terme $\sum_{\alpha_i + \beta_i =
  d_i-1} X_i^{\alpha_i} Y_i^{\beta_i}$ est une matrice bezoutienne du
jeu étalon $\uX^D$.  Son déterminant est le polynôme $\Bez(\uX,\uY) = \sum_{\alpha +
  \beta = \emouton} X^\alpha Y^\beta$.

\end{rmqs}

\begin{prop}[Quelques propriétés du bezoutien] \label{ProprietesMorley} 
\leavevmode
\begin{enumerate} [\rm i)]
\item
Pour tout $i$, on a 
$(X_i - Y_i) \Bez(\uX,\uY) \, \in \, \langle \PXminusY \rangle$.

\item
Pour tout $i$, on a 
$X_i \Bez_{d,d'} - Y_i \Bez_{d + 1, d' - 1} \in \langle \uPX, \uPY \rangle_{d+1,d'}$.

Plus généralement, pour tout $\gamma$, on a 
$X^\gamma \Bez_{d,d'} - Y^\gamma \Bez_{d + |\gamma |, d' - |\gamma |} 
\in \langle \uPX,\uPY \rangle_{d+|\gamma|,d'}$.

\item
Pour tout $|\gamma| = d'+1$, on a 
$X^\gamma \, \Bez_{d,d'} \, \in \, \langle \uPX, \uPY \rangle_{\delta +1,d'}$.

De la même façon, pour tout $|\gamma| = d+1$, on a 
$Y^\gamma \, \Bez_{d,d'} \, \in \, \langle \uPX, \uPY \rangle_{d,\delta +1}$.

En particulier, pour tout $|\gamma| = \delta+1$, les polynômes
$X^\gamma \Bez$ et $Y^\gamma \Bez$ sont dans $\langle \uPX, \uPY \rangle$.

Ainsi, 
$$
\Bez_{d,d'} \ \in \ 
\langle \uPX, \uPY \rangle_{d,d'}^\sat 
\qquad \text{et} \qquad 
\Bez \ \in \ 
\langle \uPX, \uPY \rangle^\sat 
$$

\item
Pour deux bezoutiens $\Bez^{(1)}$, $\Bez^{(2)}$ de $\uP$,
on a $\Bez^{(1)} - \Bez^{(2)} \in \langle \PXminusY \rangle$.

A fortiori $\Bez^{(1)}(\uX,\uX) = \Bez^{(2)}(\uX,\uX)$ et
$\Bez_{d,d'}^{(1)} - \Bez_{d,d'}^{(2)} \in 
\langle \uPX, \uPY \rangle_{d,d'}$. 

De plus, si $|\beta|=d' < \min(D)$, alors, dans $\bfA[\uX]$, on a
$\Bez^{(1)}_\beta - \Bez^{(2)}_\beta \in \langle \uP \rangle_d$.
\end{enumerate}
\end{prop}

\begin{proof} \leavevmode
  
Pour le point i), il suffit de multiplier l'égalité suivante 
par la transposée de la comatrice de $\dsV$ :
$$
\begin{bmatrix}
P_1(\uX)-P_1(\uY) & \cdots &  P_n(\uX)-P_n(\uY)
\end{bmatrix} 
\ = \ 
\begin{bmatrix}
X_1-Y_1 & \cdots &  X_n-Y_n
\end{bmatrix} \dsV
$$
On obtient donc 
$(X_i - Y_i) \Bez \, \in \, \langle \PXminusY \rangle$,
a fortiori $X_i \Bez - Y_i\Bez \, \in \, \langle \uPX, \uPY \rangle$.

\medskip
Le point ii) se démontre en prenant la composante bihomogène de bidegré $(d+1,d')$ 
de cette dernière appartenance. En particulier,
en réalisant les produits des égalités modulo $\langle \uPX, \uPY\rangle$ suivantes:
$$
X_j \times \Big(X_i \Bez_{d,d'} \equiv Y_i \Bez_{d+1, d'-1} \Big)
\qquad \text{et } \qquad 
Y_i \times \Big( X_j \Bez_{d+1,d'-1} \equiv Y_j \Bez_{d+2, d'-2} \Big)
$$
on obtient $X_i X_j \Bez_{d,d'} \equiv Y_i Y_j \Bez_{d+2,d'-2}$ et ainsi de suite.

\medskip
Le point iii) se déduit directement du point ii) en remarquant que $d'-|\gamma| = -1$.

\medskip
Le point~iv) : pour tout $i$, on a 
$(X_i-Y_i)\big(\Bez^{(1)} - \Bez^{(2)}\big) \in \langle \PXminusY \rangle$. 
La suite $\uX- \uY$ est $1$-sécante, d'où le résultat d'après~\ref{IndependanceNabla}.

La dernière appartenance découle de $d' < \min(D)$ et de l'égalité vue dans le préambule:
$$
\langle \uPX,\uPY \rangle_{d,d'} = \bigoplus_{|\beta|=d'} \langle \uP(\uX)\rangle_d\, Y^\beta
$$
\end{proof}

\subsubsection*{Relations entre jacobien et bezoutiens}

On rappelle la convention prise dans le chapitre
\ref{ChapJeuCirculaire} pour la matrice jacobienne d'un système
$\uP$ où $\partial_j$ est l'opération de dérivée partielle $\frac{\partial}{\partial X_j}$
(ici les $P_i$ sont en ligne):
$$
\begin{array}{c}
\\
\Jac(\uP) =
\end{array}
\setlength{\arraycolsep}{0.5\arraycolsep}
\begin {array}{cc}
   &\partial_1\quad\cdots\quad\partial_n 
\\
\begin {array}{c} P_1\\[1mm] \vdots\\[1mm] P_n\end{array}
  &\left[\begin {array}{c}  
      \\[1mm]
      \quad\partial_j(P_i) \\[1mm]
      \\[1mm]
  \end {array}\quad\right]
\\
\end {array}
$$
Son déterminant, noté $J(\uX)$, est le jacobien de $\uP$, c'est un polynôme homogène
de degré $\delta$.

\begin {prop} \leavevmode
\label {LienJacobienBezoutien}  
Soit $\Bez = \Bez(\uX,\uY)$ un bezoutien quelconque de $\uP$
et $J(\uX)$ le jacobien.

\begin {enumerate} [\rm i)]
\item
On a $J(\uX) = \Bez(\uX,\uX)$.

\item
Pour tout déterminant $\uX$-bezoutien $\nabla(\uX)$ de $\uP$, on dispose de la congruence:  
$$
J(\uX) \equiv d_1\cdots d_n\, \nabla(\uX)  \bmod \uPdelta
$$  
\end {enumerate}  
  
\end {prop}  

\begin {proof} \leavevmode

\medskip  
i) Soit $\dsV^{\rm rigid.}(\uX,\uY)$ la matrice bezoutienne rigidifiée de $\uP$
et $\Bez^{\rm rigid.}(\uX,\uY)$ son déterminant. Lorsque l'on y réalise $\uY := \uX$,
on obtient la matrice jacobienne \emph {à transposée près} à cause des conventions choisies:
$$
\dsV^{\rm rigid.}(\uX,\uX) =  \transpose{\Jac(\uP)}
$$
D'où, en prenant les déterminants, $\Bez^{\rm rigid.}(\uX,\uX) = J(\uX)$. 
On passe à un bezoutien quelconque en utilisant
$\Bez(\uX,\uX) = \Bez^{\rm rigid.}(\uX,\uX)$ d'après le point iv) de \ref{ProprietesMorley}.

\medskip
ii) Les relations d'Euler pour les polynômes homogènes $P_i$ de degré $d_i$ s'écrivent:
$$
\Jac(\uP) \begin {bmatrix} X_1 \\ \vdots \\ X_n \end {bmatrix} =
\begin {bmatrix} d_1P_1 \\ \vdots \\ d_nP_n \end {bmatrix}
$$
On en déduit $X_iJ \in \langle d_1P_1, \dots, d_nP_n\rangle$, a fortiori $J \in \uPsat_\delta$.
Comme le résultat à montrer est une identité algébrique, on peut supposer
$\uP$ générique.
D'après le théorème de Wiebe (cf.~\ref{MiniWiebe}),
il existe alors $c \in \bfA = \bfk[\indetsPi]$ tel que $J \equiv c\nabla
\bmod \uPdelta$. Comme $J$ et $\nabla$ sont homogènes en chaque $P_i$ de poids~1, on a
$c \in \bfk$. La spécialisation en le jeu étalon $\uX^D$ de la matrice jacobienne
est $\diag(d_1X_1^{d_1-1}, \dots, d_nX_n^{d_n-1})$ de déterminant $d_1\cdots d_n X^\emouton$.
Celle de $\nabla$ est $X^\emouton$ modulo $\langle\uX^D\rangle_\delta$. Il vient:
$$
d_1\cdots d_n X^\emouton \equiv c X^\emouton \bmod \langle\uX^D\rangle_\delta
$$
d'où $c = d_1 \cdots d_n$.
\end {proof}

\medskip

On déduit de~\ref{ProprietesMorley} que le polynôme
$\Bez_{d,d'}(\uX,\uX)$ est dans $\uPdelta^\sat$.  La proposition
suivante identifie sa classe dans $\vH^0_{\uX}(\bfB)_\delta$ comme un
multiple entier du déterminant $\uX$-bezoutien de $\uP$.

\begin{prop} \label{LienBezoutienNabla}
Désignons par $\chi_{d,d'}$ l'entier commun $\chi_d = \chi_{d'}$.
Considérons un bezoutien $\Bez(\uX,\uY)$ et  un déterminant $\uX$-bezoutien $\nabla(\uX)$.
Alors dans $\bfA[\uX]$:
$$
\Bez_{d,d'}(\uX,\uX) \equiv \chi_{d,d'} \nabla(\uX) \bmod \uPdelta
$$
A foriori $\Bez_{d,d'}(\uX,\uX) \equiv \Bez_{d',d}(\uX,\uX) \bmod \uPdelta$.
\end{prop}

\begin{proof}

On peut supposer $\uP$ générique d'anneau des coefficients $\bfA
= \bfk[\indetsPi]$.  On sait (depuis \ref{MiniWiebe}) que
$\bfA[\uX]_1\uPsat_\delta \subset \langle\uP\rangle_{\delta + 1}$.  La
vérification qui vient est donc inutile mais ne peut pas faire de mal.
En faisant $\uY := \uX$ dans le point ii) de \ref{ProprietesMorley}:
$$
X_i \Bez_{d,d'}(\uX,\uX) - X_i \Bez_{d+1,d'-1}(\uX,\uX) 
\ \in \ 
\langle \uPX \rangle_{\delta+1}
$$
En itérant, on obtient 
$$
X_i \Bez_{d,d'}(\uX,\uX) \in \langle\uP \rangle_{\delta+1}
$$
Nous voilà rassurés. Il existe donc $c \in \bfA$ tel que
$$
\Bez_{d,d'}(\uX,\uX) \equiv c\,\nabla(\uX)  \bmod \uPdelta
$$
On commence à connaître la musique: chacun des polynômes est homogène
de poids $1$ en les coefficients de $P_i$, donc $c \in \bfk$.  Pour
l'identifier, on spécialise en le jeu étalon (air connu).  Le polynôme
$\nabla(\uX)$ se spécialise en $X^\emouton$ modulo
$\langle\uX^D\rangle_\delta$.
Quant à $\Bez_{d,d'}(\uX,\uX)$, d'après le point iv)
de~\ref{ProprietesMorley} et la dernière remarque
de~\ref{BezoutianRemarks},  il se spécialise en un
polynôme $F \in \bfA[\uX]_\delta$ tel que
$$
F \equiv \sum_{\alpha \preccurlyeq \emouton \atop |\alpha| =d} 
X^\alpha X^{\emouton-\alpha} \bmod \langle\uX^D\rangle_\delta
$$ 
La somme à droite vaut $\chi_d\,X^\emouton$, d'où $c = \chi_d$.
\end{proof}

\begin {rmq}

Lorsque l'on somme ces congruences, on obtient la congruence modulo $\uPdelta$:
$$
\sum_{d+d' = \delta} \Bez_{d,d'}(\uX,\uX) \equiv  \Big(\sum_{d=0}^\delta \chi_d\Big)\,\nabla(\uX)
$$
Le membre de gauche n'est autre que $\Bez(\uX,\uX)$. Quant au membre de droite, la somme des
caractérisiques d'Euler-Poincaré vaut $d_1 \cdots d_n$. En effet:
$$
\sum_{d=0}^\delta \chi_d  = \dim \bfA[\uX]/\langle \uX^D\rangle
$$
Et $\bfA[\uX]/\langle \uX^D\rangle$ est un $\bfA$-module libre de
base les $(x^\alpha)_{\alpha \preccurlyeq \emouton}$, base de cardinal $d_1
\cdots d_n$.

En définitive, on obtient:
$$
\Bez(\uX,\uX) \equiv d_1 \cdots d_n\,\nabla(\uX) \bmod \uPdelta
$$
On retrouve ainsi une partie du résultat \ref{LienJacobienBezoutien}.
\end {rmq}


La proposition suivante relie les déterminants de Sylvester du chapitre précédent et
les bezoutiens.

\begin{prop} \label{LienSylvesterBezoutien}

Soit $\beta \in \bbN^n$ vérifiant $d' := |\beta| < \min(D)$.
Notons $\nabla_\beta \in \bfA[\uX]_d$ le déterminant d'une matrice $\beta$-bezoutienne
de $\uP$ (cf la définition en~\ref{MatriceBetaBezoutienne}) et
$\Bez(\uX,\uY)$ un bezoutien de~$\uP$.

\medskip
En désignant par $\Bez_{\beta'}(\uX) \in \bfA[\uX]_{\delta-|\beta'|}$ le coefficient
en $Y^{\beta'}$ de $\Bez$ vu dans $\bfA[\uX][\uY]$, on a:
$$
\Bez_\beta \equiv \nabla_\beta \bmod \langle \uP \rangle_d
$$
\end{prop}

\begin {proof} \leavevmode

On peut supposer $\uP$ générique (régulière suffit). D'après l'aspect non-dégénéré
de la multiplication (cf. \ref{SemiPairingProperty}), il suffit de montrer que pour
tout $|\gamma| = d'$, on a $X^\gamma\Bez_\beta \equiv X^\gamma \nabla_\beta \bmod \uPdelta$.

Nous savons que $\Bez(\uX,0) = \Bez_{\delta,0}$ est un déterminant $\uX$-bezoutien de~$\uP$
(cf. le premier point de la remarque \ref{BezoutianRemarks}); nous le notons $\Bez_{\underline 0}$
en cohérence avec les notations de l'énoncé.
En utilisant le théorème d'orthogonalité \ref{DualityJPJ3.10.6}, la congruence
modulo $\uPdelta$ revient à prouver celle-ci:
$$
X^\gamma \Bez_\beta \equiv \begin {cases}
\Bez_{\underline 0} &\text {si $\beta = \gamma$} \\
0      &\text{sinon} \\  
\end {cases}
$$
En utilisant le point ii) de \ref{ProprietesMorley}, les égalités
$d+|\gamma| = \delta$ et $d'-|\gamma| = 0$ et la notation $\Bez_{\underline 0}$
pour $\Bez_{\delta,0}$:
$$
X^\gamma \Bez_{d,d'} - Y^\gamma \Bez_{\underline 0} \in
\langle \uP(\uX), \uP(\uY)\rangle_{\delta,d'}
$$
Comme $d' < \min(D)$, d'après le préambule sur la bihomogénéité,
cette appartenance n'est autre que:
$$
F :=  X^\gamma \Bez_{d,d'} - Y^\gamma \Bez_{\underline 0} \in
\bigoplus_{|\beta'| = d'} \langle\uP(\uX)\rangle_\delta\, Y^{\beta'}
$$
En conséquence, chaque coefficient $F_{\beta'}$ en $Y^{\beta'}$  du polynôme $F$ vu dans $\bfA[\uX][\uY]$
appartient à $\langle\uP(\uX)\rangle_\delta$, en particulier $F_\beta \in \langle\uP(\uX)\rangle_\delta$.
Ce coefficient $F_\beta$ est  ou bien $X^\gamma \Bez_\beta - \Bez_{\underline 0}$ si $\beta = \gamma$
ou bien $X^\gamma \Bez_\beta$ si $\beta \ne \gamma$. 
L'appartenance $F_\beta \in \langle\uP(\uX)\rangle_\delta$ achève
la démonstration.
\end {proof}

\medskip

Nous allons fournir une seconde preuve plus élémentaire et, en un certain sens plus effective,
en utilisant la proposition suivante.

\begin{prop}\label{LienEffectif}
Soit $\beta$ vérifiant $|\beta| < \min(D)$ et $\dsV_\beta(\uX)$ une
matrice $\beta$-bezoutienne de~$\uP$ de déterminant $\nabla_\beta$.
Il existe une matrice bezoutienne $ \dsV(\uX,\uY)$ de~$\uP$ dont le
déterminant $\Bez(\uX,\uY)$ vérifie
$$
\Bez_\beta \, =\,  \nabla_\beta \qquad \text{dans $\bfA[\uX]$}
$$
Note: $\Bez_\beta = \Bez_\beta(\uX)$ est le coefficient en $Y^\beta$ de $\Bez$ vu dans $\bfA[\uX][\uY]$:
$$
\Bez(\uX,\uY) = \sum_{|\beta'| \leqslant \delta} \Bez_{\beta'}(\uX)Y^{\beta'}
$$
On peut prendre $\dsV$ de la forme
$$
\dsV \ = \ 
\begin{bmatrix}
\sum\limits_{p+q=\beta_1} X_1^p Y_1^q &  &  \\
 & \ddots & \\
& & \sum\limits_{p+q=\beta_n} X_n^p Y_n^q
\end{bmatrix} 
\!\times \dsV_\beta 
\ + \ \dsU
$$
où $\dsU := \dsU(\uX,\uY)$ est une matrice telle que 
$\dsU_{i,j} \in \langle \uY^{\beta+\Un} \rangle$.
\end{prop}

\begin{proof} \leavevmode

Dans le calcul suivant, la première égalité utilise la définition de $\dsV_\beta$ et la deuxième égalité utilise
l'identité $uv - u'v' = (u-u')v + (v-v')u'$:
$$
\begin{array}{rcl}
P_j(\uX) - P_j(\uY) & = & 
\displaystyle \sum_i \big( X_i^{\beta_i+1} \dsV_\beta(\uX)_{i,j} 
\ - \  Y_i^{\beta_i+1} \dsV_\beta(\uY)_{i,j}\big) \\ [0.4cm]
& = & 
\displaystyle \sum_i \big(X_i^{\beta_i+1} - Y_i^{\beta_i+1}\big) \dsV_\beta(\uX)_{i,j} 
+ \sum_\ell
\big(\dsV_\beta(\uX)_{\ell,j} - \dsV_\beta(\uY)_{\ell,j} \big) Y_\ell^{\beta_\ell+1}
\end {array}
$$
Désignons par $U_{i,\ell,j}(\uX,\uY)$ des polynômes homogènes de degré $d_j-(\beta_\ell+1)-1$ tels que
$$
\dsV_\beta(\uX)_{\ell,j} - \dsV_\beta(\uY)_{\ell,j} =
\sum_i (X_i - Y_i) \ U_{i,\ell,j}(\uX,\uY)
$$
En utilisant l'identité $a^{m+1} - b^{m+1} = (a-b) S_m(a,b)$ où
$S_m(a,b) := \sum_{p+q=m} a^pb^q$, il vient alors:
$$
P_j(\uX) - P_j(\uY) = \sum_i (X_i- Y_i) 
\big(
\, S_{\beta_i}(X_i,Y_i) \, \dsV_\beta(\uX)_{i,j}
+ \sum_{\ell} U_{i,\ell,j}(\uX,\uY) \, Y_\ell^{\beta_\ell+1}
\big)
$$
Posons $\dsU_{i,j} := \sum\limits_{\ell} U_{i,\ell,j}(\uX,\uY) \, Y_\ell^{\beta_\ell+1}$,
polynôme homogène de degré $d_j-1$. Ceci définit la matrice $\dsU$ 
et par suite la matrice $\dsV$ (via l'égalité de l'énoncé). Cette dernière est homogène
et vérifie
$$
[\uP(\uX) - \uP(\uY)] = [\uX - \uY]\,\dsV
$$
Reste à voir que le coefficient en $Y^\beta$ de $\det(\dsV)$ est $\det(\dsV_\beta)$.
Dans $\bfA[\uX][\uY]$, modulo $\langle \uY^{\beta+\Un}\rangle$:
$$
\det(\dsV) \ \equiv \  \prod_i S_{\beta_i}(X_i,Y_i) \ \times \ \det (\dsV_\beta)
\bmod \langle \uY^{\beta+\Un}\rangle
$$
En utilisant:
$$
S_{\beta_i}(X_i,Y_i) \equiv Y_i^{\beta_i} \bmod\langle\uY^{\beta+\Un}\rangle
\qquad \text{donc} \qquad
\prod_i S_{\beta_i}(X_i,Y_i) \equiv Y^\beta   \bmod\langle\uY^{\beta+\Un}\rangle
$$
on en déduit que le \emph{terme} en $Y^\beta$ de $\det(\dsV)$ 
est $\det(\dsV_\beta)\,Y^\beta$, ce qu'il fallait démontrer.
\end{proof}

\begin{proof}[Autre preuve de~\ref{LienSylvesterBezoutien}]
Soit $\dsV$ la matrice bezoutienne de la proposition~\ref{LienEffectif} 
construite à partir de la matrice $\beta$-bezoutienne $\dsV_\beta$,
qui vérifie $(\det \dsV)_\beta = \nabla_\beta$.
D'après~\ref{ProprietesMorley}-iv) et l'inégalité $|\beta| < \min(D)$, 
on a $\Bez_\beta \equiv (\det \dsV)_\beta\bmod \langle\uP\rangle_d$
c'est-à-dire $\Bez_\beta \equiv \nabla_\beta \bmod \langle\uP\rangle_d$, ce qu'il fallait démontrer.
\end{proof}


Voici une première utilisation du bezoutien: construire des
habitants de $\uPsat_d$, qui en général ne sont pas triviaux au sens
où ils n'appartiennent ni à $\langle\uP\rangle_d$ ni à
$(\ElimIdeal)\bfA[\uX]_d$ (cf l'exemple à venir).

\begin{prop}[Des formes d'inerties de degré $d$]
\label{FormesInertieDegred}
\leavevmode
\smallskip

Soit $\Bez=\Bez(\uX,\uY)$ un bezoutien de $\uP$. On écrit dans $\bfA[\uX][\uY]$:
$$
\Bez_{d,d'}(\uX,\uY) = \sum_{|\beta|=d'} \Bez_\beta(\uX)Y^\beta
\qquad \text{avec} \qquad
\Bez_\beta(\uX) \in \bfA[\uX]_d
$$
On rappelle que $s_{d'} := \dim \Jex_{1,d'}$ \og est le rang\fg{} 
de l'application de Sylvester $\Syl_{d'}=\Syl_{d'}(\uP)$.
On le note ici $r'$.  Va intervenir la transposée de $\Syl_{d'}$
que l'on peut voir de la manière suivante
$$
\transpose{\Syl_{d'}} : \bfA[\uY]_{d'} \to \bigoplus_i \bfA[\uY]_{d'-d_i} e_i
$$
\noindent
\parbox{0.7\linewidth}{   
On considère la matrice mixte ci-contre. La première ligne est
constituée des coefficients du polynôme $\Bez_{d,d'}\in \bfA[\uX][\uY]$:
ces coefficients sont dans~$\bfA[\uX]_d$.
Les autres lignes extraites de $\transpose{\Syl_{d'}}$ sont au nombre de $r'$ :
les coefficients sont dans~$\bfA$.

Alors tout mineur $\upsilon(\uX)$ d'ordre $r'+1$ de cette matrice 
appartient à $\langle \uP \rangle_d^\sat$.
Plus précisément, il vérifie:
$$
\forall\, |\gamma'|=d'+1, \quad 
X^{\gamma'}\,\upsilon(\uX)
\ \in\ \langle \uP \rangle_{\delta+1}
$$
}
\hspace{0.4cm}
\parbox{0.25\linewidth}{ 
\begin{tikzpicture}[scale = 1]
\draw[fill=gray!20]  (0,0) rectangle (4,2) ;
\path (0,0) -- (4,2) node[midway] 
        {${r' \, \mathrm{lignes} \atop \mathrm{de}\, \transpose {\Syl_{\!d'}} }$} ;
\draw[fill=gray!60]  (0,2) rectangle (4,2.5) ;
\path (0,2) -- (4,2.5) node[midway] {$\scriptstyle\big(\Bez_{\beta}(\uX)\big)_{|\beta|=d'}$} ;
\path (0,2.8) -- (4,2.8) node[midway] {$\scriptstyle \bfA[\uY]_{d'}$} ;
\end{tikzpicture}
}
\end{prop}

\begin{proof}
On commence par prouver un résultat un peu plus général qui n'est pas lié au bezoutien $\Bez$.

\smallskip

\noindent
\parbox{0.7\linewidth}{

Soit $F(\uX, \uY)$ un polynôme bihomogène appartenant à $\langle \uPX,
\uPY \rangle_{c,d'}$ que l'on écrit $F = \sum_{|\beta | = d'}
F_{\beta}(\uX) Y^\beta$.  Nous allons montrer que les mineurs maximaux
(d'ordre $r'+1$) de la matrice ci-contre sont dans $\langle \uPX
\rangle_c$.

Le polynôme $F$ est une somme de polynômes élémentaires du type 
$$
\left\{
\begin{tabular}{ll}
$H_{\beta}(\uX)\, Y^\beta$ & 
avec $H_\beta(\uX) \in \langle \uPX \rangle_c$ et $|\beta | = d'$ \\ [0.2cm]
$X^\gamma \, H_\gamma(\uY)$ & 
avec $H_\gamma(\uY) \in \langle \uPY \rangle_{d'}$ et $|\gamma| = c$ 
\end{tabular}
\right.
$$
}
\hspace{0.4cm}
\parbox{0.25\linewidth}{ 
\begin{tikzpicture}[scale = 1]
\draw[fill=gray!20]  (0,0) rectangle (4,2) ;
\path (0,0) -- (4,2) node[midway] 
        {${r' \, \mathrm{lignes} \atop \mathrm{de}\, \transpose {\Syl_{\!d'}} }$} ;
\draw[fill=gray!60]  (0,2) rectangle (4,2.5) ;
\path (0,2) -- (4,2.5) node[midway] {$\scriptstyle\big(F_{\beta}(\uX)\big)_{|\beta|=d'}$} ;
\path (0,2.8) -- (4,2.8) node[midway] {$\scriptstyle\bfA[\uY]_{d'}$} ;
\end{tikzpicture}
}

Par linéarité du déterminant par rapport à la première ligne, il suffit de montrer 
le résultat pour $F$ polynôme élémentaire ci-dessus.

\noindent
$\rhd$ 
Dans le premier cas, la première ligne est $(0, \dots, 0, H_\beta(\uX), 0, \dots, 0)$ 
avec $H_\beta(\uX) \in \langle \uP \rangle_c$.
Par conséquent, un mineur d'ordre $r'+1$ est nul ou multiple de 
$H_\beta(\uX) \in \langle \uPX \rangle_c$.

\noindent
$\rhd$ 
Dans le deuxième cas, la première ligne est, à $X^\gamma$ près, une combinaison-$\bfA$-linéaire 
de lignes de $\transpose \Syl_{d'}$, car $H_\gamma(\uY) \in \langle \uPY \rangle_{d'}
= \Im \Syl_{d'}$. Or tous les mineurs d'ordre $r'+1$ de $\transpose \Syl_{d'}$ sont nuls.

\smallskip
Dans les deux cas, tout mineur d'ordre $r'+1$ de la matrice en question est dans 
$\langle \uPX \rangle_c$.

\smallskip

Il suffit d'appliquer ce résultat en prenant $c=\delta+1$ et $F = X^{\gamma'} \Bez_{d,d'}$, pour $\gamma'$ fixé 
tel que $|\gamma'|=d'+1$.
C'est licite, car $X^{\gamma'} \Bez_{d,d'} \in \langle \uPX, \uPY \rangle_{\delta+1,d'}$
d'après~\ref{ProprietesMorley}.
\end{proof}

\medskip

\noindent
{\bf Exemple du format $D = (1,2,2)$ de degré critique $\delta=2$, $d=1$ (donc $d'=1$)}

\medskip
Soit $\uP = (P_1, P_2, P_3)$ le système de format $D$:
$$
P_1 = a_1X_1 + a_2X_2 + a_3 X_3, \qquad
P_2 = b_1X_1^2 + b_2X_1X_2 + b_3X_1X_3 + b_4X_2^2 + b_5X_2X_3 + b_6X_3^2
$$
Le polynôme $P_3$ est analogue à $P_2$ avec des coefficients $c_\bullet$ au lieu de $b_\bullet$.
La matrice bezoutienne rigidifiée de $\uP$ est la suivante:
$$
\dsV = 
\begin {bmatrix}
a_1 & b_1(X_1+Y_1) + b_2X_2 + b_3X_3  & c_1(X_1+Y_1) + c_2X_2 + c_3X_3 \\
a_2 & b_2Y_1 + b_4(X_2+Y_2) + b_5X_3  & c_2Y_1 + c_4(X_2+Y_2) + c_5X_3 \\
a_3 & b_3Y_1 + b_5Y_2 + b_6(X_3+Y_3)  & c_3Y_1 + c_5Y_2 + c_6(X_3+Y_3) \\
\end {bmatrix}
$$
Pour les coefficients en $\uY$ de son déterminant $\Bez = \Bez(\uX,\uY)$
(ce sont des polynômes en $\uX$), on adopte
une \og indexation par le bas\fg{} à ne pas confondre avec l'indexation des
composantes bihomogènes de $\Bez$:
$$
\Bez = \Bez_{2,0} + \Bez_{1,1} + \Bez_{0,2}
\qquad
\left\{
\begin {array} {lcl} 
\Bez_{2,0}(\uX) &=& \Bez_0(\uX) \\
\Bez_{1,1}(\uX,\uY) &=& \sum_i \Bez_i(\uX)Y_i \\
\Bez_{0,2}(\uY) &=& \sum_{i\leqslant j} \Bez_{ij} Y_iY_j \\
\end {array}
\right.
$$
Quelques précisions. Tout d'abord, une propriété générale découlant de
$[\uP(\uX) - \uP(\uY)] = [\uX-\uY]\,\dsV$: les polynômes
$\Bez_{\delta,0}(\uX)$ et $\Bez_{0,\delta}(\uX)$, homogènes de degré
$\delta$, sont des déterminants $\uX$-bezoutien de $\uP$, donc
$$
\Bez_{0,\delta}(\uX) \equiv \Bez_{\delta,0}(\uX)  \bmod \langle\uP\rangle_\delta,
$$
Cette congruence générale peut aussi se déduire du fait que $\Bez(\uY,\uX)$
est également un bezoutien de $\uP$ et du point iv) de \ref{ProprietesMorley}, donc:
$$
\Bez_{d,d'}(\uY,\uX) \equiv \Bez_{d',d}(\uX,\uY) \bmod \langle \uP(\uX),\uP(Y) \rangle_{d,d'}
$$
En prenant $d=0$ donc $d'=\delta$, on retrouve la congruence précédente.

\medskip

Dans l'exemple, les $\Bez_i$ sont des formes linéaires en $X_1, X_2, X_3$ et
on a les congruences:
$$
\Bez_{0,2}(\uX) \equiv \Bez_{2,0}(\uX)  \bmod \langle\uP\rangle_2,
\qquad\quad
\Bez_{1,1}(\uX,\uX) \equiv 2\Bez_{2,0}(\uX) \bmod \langle\uP\rangle_2
$$
Avec $d=1$, on a $\Jex_{1,d} = \bfA X_1$ donc $\chi_d = 2$, ce qui
explique le facteur~$2$ à droite, en vertu
de~\ref{LienBezoutienNabla}.  Comme $d'=1$, on a $r' \overset{\rm
def}{=} \dim \Jex_{1,d'} = 1$.  La matrice mixte dont il est question
en~\ref{FormesInertieDegred} est:
$$
\begin {bmatrix}
\Bez_1    &  \Bez_2     & \Bez_3 \\ 
a_1       &     a_2     & a_3  \\
\end {bmatrix}
$$
d'où l'obtention, via ses mineurs d'ordre 2, de formes linéaires
$\ell_{i,j} := a_i\Bez_j - a_j\Bez_i \in \uPsat_d$.  On donne le
résultat suivant (dans le cas où $\uP$ est générique) sans justifier
l'égalité du milieu (en degré 1):
$$
\vH^0_{\uX}(\bfB)_0 = \bfA\Res(\uP), \qquad
\vH^0_{\uX}(\bfB)_1 = \bfA\ell_{12} + \bfA\ell_{13} + \bfA\ell_{23}, \qquad
\vH^0_{\uX}(\bfB)_2 = \bfA\Bez_0
$$

\newcommand \vabsM {3.7}
\newcommand \vordM {2.4}
\newcommand \absM {3.7}
\newcommand \ordM {2.4}
\newcommand \epaisseur {0.4}
\newcommand \taille {8}

\newcommand \milieu {(\absM, \ordM)}
\newcommand \TransRestSyl[1] {\hbox{${}^t\Big({\mathrm{restr.} \atop \mathrm{de}\,\Syl_{#1}}\Big)$}}
\newcommand \RestSyl[1]  {${\mathrm{restr.} \atop \mathrm{de}\,\Syl_{#1}}$}
\newcommand \coinH {(\absM, \taille-\absM)} 
\newcommand \coinB {(\taille-\ordM ,\ordM)}

\newcommand \NorthUn {(0,\taille) -- (\absM, \taille)} 
\newcommand \NorthDeux {(\absM,\taille)--(\taille,\taille)} 
\newcommand \NorthMilieu {(\absM,\taille) -- (\taille-\ordM, \taille)} 
\newcommand \NorthDroite {(\taille-\ordM, \taille) -- (\taille, \taille)}
\newcommand \EastUn {(\taille,\taille) -- (\taille, \ordM)} 
\newcommand \EastHaut {(\taille, \taille) -- (\taille, \taille-\absM)}
\newcommand \EastMilieu {(\taille, \taille-\absM) -- (\taille, \ordM)}
\newcommand \EastDeux {(\taille, \ordM) -- (\taille,0)}
\subsection{Construction $\mho_{d,\protect\uP} : \bfA[\protect\uX,\protect\uY]_{d,d'}
   \rightarrow \End\big(\bfA[\protect\uX]_d\oplus\Jex_{1,d'}^\star\big)$ pilotée par $\minDiv$}

Sous le couvert d'un système $\uP$ de format $D$, 
les auteurs ont maintenant l'intention d'associer à un polynôme bihomogène $F \in \bfA[\uX,\uY]_{d,d'}$ 
un \textit{endomorphisme} du $\bfA$-module monomial $\bfA[\uX]_d \oplus \Jex_{1,d'}^\star$
ce qui leur permettra de considérer un scalaire \textit{très précis} (et pas défini au signe près)
à savoir le déterminant de cet endomorphisme.
Certains auteurs se contentent de définir des matrices
dans des bases monomiales parfois non précisées ou alors très approximatives 
rendant les déterminants définis à un signe près, 
ce qui est bien sûr problématique dans le cadre du résultant.
Cette volonté de définir un \textit{endomorphisme} nécessite la petite gymnastique qui vient.
Afin de ne pas effrayer les lecteurs, nous avons tenu à accompagner cette gymnastique 
de quelques jolis dessins.

\medskip

Quelques rappels sont nécessaires pour mieux comprendre qui dépend de
$D$, de $\uP$ et du mécanisme de sélection $\minDiv$.  Tout d'abord,
on dispose de la décomposition monomiale $\bfA[\uX]_d = \Jex_{1,d}
\oplus \Smac_{0,d}$ qui ne dépend que de $D$.  Par ailleurs, lié à
$\minDiv$, il y a l'isomorphisme $\varphi$ de $\Jex_{1,d}\subset
\bfA[\uX]_d$ sur $\Smac_{1,d} \subset \rmK_{1,d}$:
$$
\varphi :\ \Jex_{1,d} \ \overset{\simeq}{\longrightarrow} \  \Smac_{1,d}, 
\qquad 
X^\alpha \ \longmapsto\ \dfrac{X^\alpha}{X_i^{d_i}} e_i 
\quad \text{où $i = \minDiv(X^\alpha)$}
$$
Pour mettre en place l'endomorphisme $\mho_{d}(F) = \mho_{d,\uP}(F)$,  
trois modules monomiaux isomorphes vont coopérer. Nous leur donnons le nom
\emph{local} ci-dessous; au milieu, nous avons placé $B$ (pour but): nous visons à
définir un endomorphisme de $B$, les modules $A,C$ étant en quelque sorte
des auxiliaires
$$
A = \Smac_{1,d} \oplus \bfA[\uY]_{d'}^\star, \qquad
B = \bfA[\uX]_d \oplus \Jex_{1,d'}^\star, \qquad
C = \bfA[\uX]_d \oplus \Smac_{1,d'}^\star
$$
Nous allons d'abord définir $\bbS_d(F) = \bbS_{d,\uP}(F)$ comme schématisé ci-dessous
et ensuite $\mho_d(F)$ sera obtenu en composant par des isomorphismes monomiaux qui ne
dépendent que de $(D,\minDiv)$, pas de $\uP$:
$$
\mho_d(F) : 
\xymatrix @C=2cm {
B \ar[r]^-{\textstyle \Phi^{\mathrm{right}}}_{\simeq} & A\ar[r]^{\textstyle \bbS_d(F)} &
C \ar[r]^-{\textstyle \Phi^{\mathrm{left}}}_{\simeq} & B
}
$$
Dans la définition géante qui vient, il sera mention de $\Phi = \Phi^{\mathrm{right}} \circ
\Phi^{\mathrm{left}} : C \to A$. Interviendront également
$\Syl_d = \Syl_d(\uP) : \rmK_{1,d} \to \bfA[\uX]_d$, sa restriction à $\Smac_{1,d}$,
sa co-restriction $\beta_{1,d} = \beta_{1,d}(\uP) : \Smac_{1,d} \to \Jex_{1,d}$,
l'endomorphisme $W_{1,d} = W_{1,d}(\uP)$ de $\Jex_{1,d}$
et, \emph {à transposée près}, les analogues pour $d'$ à la place de $d$.

\begin{defn} [Dite \og définition géante\fg]
\label{DefGeante}
Pour $F(\uX,\uY) = 
\displaystyle \sum_{|\alpha | = d \atop |\beta| =d'} a_{\alpha\beta}X^\alpha Y^\beta 
\in \bfA[\uX, \uY]_{d,d'}$, 
on note $F^\dagger$ l'application linéaire définie sur
la base monomiale par 
$$
F^\dagger : 
\bfA[\uY]_{d'}^\star \ \longrightarrow \  \bfA[\uX]_d  
\qquad \qquad
(Y^\beta)^\star \ \longmapsto \ \sum_{|\alpha| = d} a_{\alpha \beta}X^\alpha
$$
\noindent
L'application linéaire $\bbS_d(F)$, premier maillon, peut ainsi se mettre en place ; sa définition est schématisée ci-dessous :
$$
\bbS_d(F) : \ \Smac_{1,d} \oplus \bfA[\uY]_{d'}^\star \ \longrightarrow \ 
\bfA[\uX]_d \oplus \Smac_{1,d'}^\star
$$

\centerline{
\begin{tikzpicture}[scale = 0.5]
\draw[fill=gray!20]  (0,\taille) rectangle \milieu ;
\path (0,\taille) -- \milieu node[midway] {\RestSyl{d}} ;
\draw[fill=gray!60]  (\taille,\taille) rectangle \milieu ;
\path (\taille,\taille) -- \milieu node[midway] {$F^\dagger$} ; 
\draw[fill=gray!20]  (\taille, 0) rectangle \milieu ;
\path (\taille,0) -- \milieu node[midway] {\TransRestSyl{d'}} ;
\draw  (0,0) rectangle \milieu ;
\path \NorthUn node[midway,above] {$\Smac_{1,d}$} ;
\path \NorthDeux node[midway,above] {$\bfA[\uY]_{d'}^\star$} ;
\path \EastUn node[midway,right] {$\bfA[\uX]_d$} ;
\path \EastDeux node[midway,right] {$\Smac_{1,d'}^\star$} ;
\path (0,0) -- (0, \taille) node[midway, left] {$\bbS_d(F) \ = \ $} ;
\end{tikzpicture}
\begin{tikzpicture}[scale = 0.5]
\draw[fill=gray!20]  (0,\taille) rectangle \milieu ;
\draw[fill=gray!60]  (\taille,\taille) rectangle \milieu ;
\path (\taille,\taille) -- \milieu node[midway] {$F^\dagger$} ; 
\draw[fill=gray!20]  (\taille, 0) rectangle \milieu ;
\draw  (0,0) rectangle \milieu ;
\draw[thick, dashed] (0,\taille) rectangle \coinH ;
\path (0,\taille) -- \coinH node[midway] {$\scriptstyle\beta_{1,d}$} ; 
\draw[thick, dashed] \coinB rectangle (\taille, 0) ;
\path \coinB -- (\taille, 0) node[midway] {$\scriptstyle\transpose\beta_{1,d'}$} ; 
\path \NorthUn node[midway,above] {$\Smac_{1,d}$} ;
\path \NorthMilieu node[midway,above] {$\Smac_{0,d'}^\star$} ;
\path \NorthDroite node[midway,above] {$\Jex_{1,d'}^\star$} ;
\path \EastHaut node[midway,right] {$\Jex_{1,d}$} ;
\path \EastMilieu node[midway,right] {$\Smac_{0,d}$} ;
\path \EastDeux node[midway,right] {$\Smac_{1,d'}^\star$} ;
\path (0,0) -- (0, \taille) node[midway, left] {
\begin{tabular}{l} \hspace{0.3cm} en raffinant \\ la décomposition,\hspace{1cm} \end{tabular}} ;
\end{tikzpicture}
}

\noindent
\parbox{0.63\linewidth}{
L'isomorphisme monomial $\Phi = \varphi \oplus {\mouton}^{\rm sw} \oplus \transpose \varphi$
réalise un isomorphisme de l'espace d'arrivée de $\bbS_d(F)$ sur son espace de départ
$$
\Phi : \ 
(\Jex_{1,d} \oplus \Smac_{0,d}) \oplus \Smac_{1,d'}^\star
\ \longrightarrow \ 
\Smac_{1,d} \oplus (\Smac_{0,d'}^\star \oplus \Jex_{1,d'}^\star)
$$ 
}
\parbox{0.35\linewidth}{  
\begin{tikzpicture}[scale = 0.4]
\draw  (0,0) rectangle (\taille, \taille) ;
\draw[thick, dashed, fill=gray!20] (0,\taille) rectangle \coinH ;
\path (0,\taille) -- \coinH node[midway] {\hbox{\scriptsize $\varphi$}} ;
\draw \coinH rectangle \coinB ;
\path \coinH -- \coinB node[midway] {$\mouton^{\rm sw}$} ;
\draw[thick, dashed,fill=gray!20] \coinB rectangle (\taille, 0) ;
\path \coinB -- (\taille, 0) node[midway] {\hbox{\scriptsize $\transpose \varphi$}} ;
\path \NorthUn node[midway,above] {$\Jex_{1,d}$} ;
\path \NorthMilieu node[midway,above]  {$\Smac_{0,d}$} ;
\path \NorthDroite node[midway,above] {$\Smac_{1,d'}^\star$} ;
\path \EastHaut node[midway,right] {$\Smac_{1,d}$} ;
\path \EastMilieu node[midway,right] {$\Smac_{0,d'}^\star$} ;
\path \EastDeux node[midway,right] {$\Jex_{1,d'}^\star$} ;
\path (0,0) -- (0, \taille) node[midway, left] {$\Phi \ = \ $} ;
\end{tikzpicture}
}

\bigskip

Comme annoncé, $\Phi$ est la composée des deux isomorphismes monomiaux:
$$
\Phi^{\mathrm{right}} : \ 
(\Jex_{1,d}\oplus\Smac_{0,d}) \oplus \Jex_{1,d'}^\star
\, \rightarrow \, 
\Smac_{1,d} \oplus (\Smac_{0,d'}^\star \oplus \Jex_{1,d'}^\star),
\quad 
\Phi^{\mathrm{left}} : \ 
(\Jex_{1,d} \oplus \Smac_{0,d}) \oplus \Smac_{1,d'}^\star
\, \rightarrow \,
(\Jex_{1,d} \oplus \Smac_{0,d}) \oplus \Jex_{1,d'}^\star
$$ 

\begin{tikzpicture}[scale = 0.5]
\draw  (0,0) rectangle (\taille, \taille) ;
\draw[thick, dashed, fill=gray!20] (0,\taille) rectangle \coinH ;
\path (0,\taille) -- \coinH node[midway] {\hbox{\scriptsize $\varphi$}} ;
\draw \coinH rectangle \coinB ;
\path \coinH -- \coinB node[midway] {$\mouton^{\rm sw}$} ;
\draw[thick, dashed] \coinB rectangle (\taille, 0) ;
\path \coinB -- (\taille, 0) node[midway] {\hbox{\scriptsize $\id$}} ;
\path \NorthUn node[midway,above] {$\Jex_{1,d}$} ;
\path \NorthMilieu node[midway,above]  {$\Smac_{0,d}$} ;
\path \NorthDroite node[midway,above] {$\Jex_{1,d'}^\star$} ;
\path \EastHaut node[midway,right] {$\Smac_{1,d}$} ;
\path \EastMilieu node[midway,right] {$\Smac_{0,d'}^\star$} ;
\path \EastDeux node[midway,right] {$\Jex_{1,d'}^\star$} ;
\path (0,0) -- (0, \taille) node[midway, left] {$\Phi^{\mathrm {right}}  \ = \ $} ;
\end{tikzpicture}
\hspace{1cm}
\begin{tikzpicture}[scale = 0.5]
\draw  (0,0) rectangle (\taille, \taille) ;
\draw[thick, dashed] (0,\taille) rectangle \coinH ;
\path (0,\taille) -- \coinH node[midway] {\hbox{\scriptsize $\id$}} ;
\draw \coinH rectangle \coinB ;
\path \coinH -- \coinB node[midway] {\hbox{\scriptsize $\id$}} ;
\draw[thick, dashed,fill=gray!20] \coinB rectangle (\taille, 0) ;
\path \coinB -- (\taille, 0) node[midway] {\hbox{\scriptsize $\transpose \varphi$}} ;
\path \NorthUn node[midway,above] {$\Jex_{1,d}$} ;
\path \NorthMilieu node[midway,above]  {$\Smac_{0,d}$} ;
\path \NorthDroite node[midway,above] {$\Smac_{1,d'}^\star$} ;
\path \EastHaut node[midway,right] {$\Jex_{1,d}$} ;
\path \EastMilieu node[midway,right] {$\Smac_{0,d}$} ;
\path \EastDeux node[midway,right] {$\Jex_{1,d'}^\star$} ;
\path (0,0) -- (0, \taille) node[midway, left] {$\Phi^{\mathrm{left}} \ = \ $} ;
\end{tikzpicture}

\bigskip

\noindent
\parbox{0.63\linewidth}{
Ils  permettent d'obtenir l'\textrm{endomorphisme} convoité $\mho_d(F)$
de $\bfA[\uX]_d \oplus \Jex_{1,d'}^\star$ par la formule :
$$
\mho_d(F) \ = \ 
\Phi^{\mathrm{left}}\,\circ\, \bbS_d(F) \,\circ \, \Phi^{\mathrm{right}}
$$
Coucou les matrices $W_{1,d}$ et $\transpose{W_{1,d'}}$!
}
\parbox{0.35\linewidth}{  
\begin{tikzpicture}[scale = 0.4]
\draw[fill=gray!20]  (0,\taille) rectangle \milieu ;
\draw[fill=gray!60]  (\taille,\taille) rectangle \milieu ;
\draw[fill=gray!20]  (\taille, 0) rectangle \milieu ;
\draw  (0,0) rectangle \milieu ;
\draw[thick, dashed, color=blue] (0,\taille) rectangle \coinH ;
\path (0,\taille) -- \coinH node[midway] {\hbox{\scriptsize $W_{1,d}$}} ;
\draw[thick, dashed, color=blue] \coinB rectangle (\taille, 0) ;
\path \coinB -- (\taille, 0) node[midway] {\hbox{\scriptsize $\transpose W_{1,d'}$}} ;
\draw[thick, dashed, color=blue] \coinH rectangle \coinB ;
\path \coinH rectangle \coinB node [midway] {$F_{\hbox{\tiny\mouton}}^\dagger$} ;
\path \EastHaut node[midway,right] {$\Jex_{1,d}$} ;
\path \EastMilieu node[midway,right] {$\Smac_{0,d}$} ;
\path \EastDeux node[midway,right] {$\Jex_{1,d'}^\star$} ;
\path (0,0) -- (0, \taille) node[midway, left] {$\mho_d(F) = $} ;
\end{tikzpicture}
}

\noindent
Précisément, $\mho_d(F)$ réalise sur la base monomiale:
$$
\begin{array}[t]{rcl}
(\Jex_{1,d} \oplus \Smac_{0,d}) \oplus \Jex_{1,d'}^\star 
& \longrightarrow & \bfA[\uX]_d \,\oplus\, \Jex_{1,d'}^\star
\\ [0.4cm]
(X^\alpha \oplus X^\gamma) \oplus (Y^\beta)^\star &\longmapsto & 
\big(\Syl_d \circ \varphi\big)(X^\alpha) + F^\dagger(Z)
\ \oplus \ 
\big(\transpose \varphi \circ \transpose \Syl_{d'}\big)(Z)
\\ [0.4cm]
         &&\hspace {1cm} \text{avec } Z = (Y^{\emouton-\gamma})^\star + (Y^\beta)^\star \in \bfA[\uY]_{d'}^\star
\end{array}
$$
Enfin, pour $\sigma \in \fS_n$, on peut définir de manière analogue $\mho_d^\sigma$, à l'aide de la 
décomposition de Macaulay tordue par $\sigma$.
\end{defn}

\label{NOTA17-mho}%

\medskip

\begin{rmq} [Pour les amateurs d'isomorphismes monomiaux]

Après cette définition géante, on peut en ajouter une au panier en notant
$\calL_{d,d'}$ le module monomial $\bfA[\uX]_d \oplus
\Jex_{1,d'}^\star$ sur lequel opère~$\mho_d(F)$. On dispose alors
d'un isomorphisme monomial canonique
$$
\calL_{d,d'} \ \simeq \ (\calL_{d',d})^\star
$$
A cet effet, on écrit:
$$
\calL_{d,d'} = \Jex_{1,d} \oplus \Smac_{0,d} \oplus \Jex_{1,d'}^\star,
\qquad
\calL_{d',d} = \Jex_{1,d'} \oplus \Smac_{0,d'} \oplus \Jex_{1,d}^\star,
\qquad
(\calL_{d',d})^\star = \Jex_{1,d'}^\star \oplus \Smac_{0,d'}^\star \oplus \Jex_{1,d}
$$
En respectant l'ordre des facteurs ci-dessus, l'isomorphisme canonique est :
$$
X^\alpha \oplus X^\gamma \oplus (Y^\beta)^\star \longmapsto
(Y^\beta)^\star  \oplus (Y^{\emouton-\gamma})^\star \oplus X^\alpha
$$
Cet isomorphisme aura son importance dans l'histoire $(\uX,\uY)$ versus $(\uY,\uX)$,
cf la proposition~\ref{PermutationXYdansMho}, et pourquoi pas toute la section
contenant ce résultat.
\end{rmq}

\subsubsection*{Le cas particulier de $\delta-\min(D) < d \leqslant \delta$: adéquation entre
$\mho_d$ et $\swbsOmega_d$}

Pour un entier quelconque $d$, en notant $\calS_{0,d}$ la base monomiale de $\Smac_{0,d}$, nous avons défini
dans la section~\ref{CanonicalCayleyDetSection} un mécanisme, piloté par $\minDiv$, permettant d'élaborer
des \emph{endomorphismes} de $\bfA[\uX]_d$ (voir la définition~\ref{OmegaDetNumerator}):
$$
\bsOmega_d = \bsOmega_{d,\uP}  : \AXdSod \to \End(\bfA[\uX]_d)
$$
Nous l'avons accompagné de plusieurs exemples et de deux dessins
grisés (cf. page~\pageref{DessinOmegad}) analogues à ceux figurant dans la
définition géante, à l'exception près que dans ces deux dessins, il n'y
pas de \og bande horizontale\fg{} correspondant à $\Jex_{1,d'}^\star$.
Nous avons également fourni la formule suivante exprimant
la forme $r_{0,d}$-linéaire alternée canonique $\Det_d : \AXdSod \to
\bfA$ à l'aide de~$\bsOmega_d$:
$$
\Det_d(f) = \dfrac{\det\bsOmega_{d}(f)}{\Delta_{2,d}} =
\dfrac{\det\bsOmega_{d}(f)}{\det W_{2,d}}
$$
Cette forme passe au quotient modulo $\langle\uP\rangle_d$ pour définir
$\Det_d : \BdSod \to \bfA$.

\medskip
Dans le chapitre précédent, nous avons étudié le cas où $d$ est dans
la fourchette $\delta - \min(D) < d \leqslant \delta$ ce qui équivaut à $0 \leqslant d' < \min(D)$. On
a alors $\Jex_{1,d'} =0$ donc $\Smac_{0,d'} = \bfA[\uX]_{d'}$. En composant avec l'isomorphisme
mouton-swap $\bfA[\uX]_d^{\calS_{0,d'}} \overset{\simeq}{\longrightarrow} \AXdSod$, nous avons réalisé
(cf. \ref{swNotation}) une \og copie\fg{} $\swbsOmega_d$ de~$\bsOmega_d$:
$$
\swbsOmega_d = \swbsOmega_{d,\uP}  : \bfA[\uX]_d^{\calS_{0,d'}} \to \End(\bfA[\uX]_d)
$$
Et nous avons obtenu (cf. théorème \ref{ResViaDetd}) une expression du résultant à l'aide
de la famille $\bsnabla_{d'} = (\nabla_\beta)_{|\beta| = d'}$ des déterminants de Sylvester
de $\uP$:
$$
\Res(\uP) = \swDet_d\big(\bsnabla_{d'})  =
\dfrac{\det\swbsOmega(\bsnabla_{d'})}{\det W_{2,d}}
\leqno(\star)
$$
L'examen des définitions de $\bsOmega_d$ et $\mho_d$ dans la
fourchette $\delta - \min(D) < d \leqslant \delta$, contrainte pour laquelle
on a $\bfA[\uX]_d \oplus \Jex_{1,d'}^\star = \bfA[\uX]_d$, conduit à
l'égalité \boxed{\swbsOmega_d = \mho_d} modulo l'identification
$$
\bfA[\uX]_d^{\calS_{0,d'}}  \simeq \bfA[\uX,\uY]_{d,d'}
$$
Cette identification, qui  utilise le fait que $\calS_{0,d'}$ est la base monomiale
de $\bfA[\uX]_{d'}$, consiste à rapprocher $f \in \bfA[\uX]_d^{\calS_{0,d'}}$ du polynôme
$F \in \bfA[\uX,\uY]_{d,d'}$ défini par:
$$
F(\uX,\uY) = \sum_{|\beta| = d'} f(X^\beta)\,Y^\beta
$$
Enfin, dans la proposition \ref{LienSylvesterBezoutien}, 
pour $|\beta| < \min(D)$, nous avons relié déterminant de Sylvester~$\nabla_\beta(\uX)$ et
composante $\Bez_\beta(\uX)$ de n'importe quel bezoutien $\Bez(\uX,\uY)$ de $\uP$:
$$
\nabla_\beta(\uX) \equiv \Bez_\beta(\uX)  \bmod \langle\uP\rangle_d
$$
Vu l'adéquation entre $\swbsOmega_d$ et $\mho_d$ et le fait que
$\swDet_d$ passe au quotient modulo $\langle\uP\rangle_d$, on peut reformuler
l'égalité~$(\star)$ à l'aide des nouvelles notations:
$$
\Res(\uP) = \dfrac{\det\mho_d(\Bez_{d,d'})}{\det W_{2,d}}
$$
Elle est pas belle la vie?

\begin{exemple}

Ici $n=3$, $D=(5,1,2)$, $\delta=5$, $d=3$ donc $d'=2$. Avouons que nous tenons
à voir avec nos yeux les matrices $W_{1,d}$ et $\transpose{W_{1,d'}}$
schématisées dans la définition géante.  

Pour des raisons typographiques, nous allons utiliser $\bfA[X,Y,Z]$ au
lieu de $\bfA[\uX] = \bfA[X_1,X_2,X_3]$ et $\bfA[x,y,z]$ au lieu de
$\bfA[\uY] = \bfA[Y_1,Y_2,Y_3]$. Un polynôme $F \in \bfA[X,Y,Z;x,y,z]_{3,2}$
est écrit sous la forme
$$
F = F_{x^2} x^2 + F_{xy} xy + \cdots + F_{z^2}z^2, \qquad
F_{x^2}, F_{xy}, \dots \in \bfA[X,Y,Z]_3
$$
Le système $\uP$ de format $D$ est:
$$
P_2 = b_1X_1 + b_2X_2 + b_3X_3,
\qquad\quad
P_3 = c_1X_1^2 + c_2X_1X_2 + c_3X_1X_3 + c_4X_2^2 + c_5X_2X_3 + c_6X_3^2
$$
Inutile de montrer $P_1$, il n'intervient pas dans $\bbS_3(\sbullet)$.
Voici, dans les bases indiquées, la matrice de $\bbS_3(F)$  où $F^\dagger$ n'est visualisé
que partiellement, par exemple:
$$
F_{z^2} = aX^3 + bX^2Y + \cdots + jZ^3
$$
$$
\bbS_3(F) =
\NorthEastBordermatrix{
\Veti{X^{2}\,e_{2}} &\Veti{XY\,e_{2}} &\Veti{XZ\,e_{2}} &\Veti{Y^{2}\,e_{2}} &\Veti{YZ\,e_{2}} &\Veti{Z^{2}\,e_{2}} &\Veti{X\,e_{3}} &\Veti{Z\,e_{3}}
   &\Veti{x^{2}} & \Veti{xy} & \Veti{xz} & \Veti{y^{2}} & \Veti{yz} & \Veti{z^{2}} & \\
b_{1} & . & . & . & . & . & c_{1} & . &\sbullet   &   &   &   &\sbullet  & a & \Heti{X^{3}} \\
b_{2} & b_{1} & . & . & . & . & c_{2} & . &   &   &   &   &          & b & \Heti{X^{2}Y} \\
b_{3} & . & b_{1} & . & . & . & c_{3} & c_{1} &   &   &   &   &      & c & \Heti{X^{2}Z} \\
. & b_{2} & . & b_{1} & . & . & c_{4} & . &   &   &   &   &          & d & \Heti{XY^{2}} \\
. & b_{3} & b_{2} & . & b_{1} & . & c_{5} & c_{2} &   &   &   &   &  & e & \Heti{XYZ} \\
. & . & b_{3} & . & . & b_{1} & c_{6} & c_{3} &   &   &   &   &      & f & \Heti{XZ^{2}} \\
. & . & . & b_{2} & . & . & . & . &   &   &   &   &                & g & \Heti{Y^{3}} \\
. & . & . & b_{3} & b_{2} & . & . & c_{4} &   &   &   &   &         & h & \Heti{Y^{2}Z} \\
. & . & . & . & b_{3} & b_{2} & . & c_{5} &   &   &   &   &         & i & \Heti{YZ^{2}} \\
. & . & . & . & . & b_{3} & . & c_{6} &\sbullet   &   &   &   &\sbullet     & j & \Heti{Z^{3}} \\
. & . & . & . & . & . & . & . & b_{1} & b_{2} & b_{3} & . & . & . & \Heti{x\,e_{2}} \\
. & . & . & . & . & . & . & . & . & b_{1} & . & b_{2} & b_{3} & . & \Heti{y\,e_{2}} \\
. & . & . & . & . & . & . & . & . & . & b_{1} & . & b_{2} & b_{3} & \Heti{z\,e_{2}} \\
. & . & . & . & . & . & . & . & c_{1} & c_{2} & c_{3} & c_{4} & c_{5} & c_{6} & \Heti{e_{3}} \\
}
$$
\noindent 
En haut de la matrice, on voit une base monomiale de l'espace de départ de $\bbS_3$ à savoir $\Smac_{1,d} \oplus \bfA[x,y,z]_{d'}^\star$, et à droite, on voit une base monomiale de l'espace d'arrivée $\bfA[X,Y,Z]_d \oplus \Smac_{1,d'}^\star$.
On identifie ces deux espaces avec $\Phi$, et on obtient la matrice de l'endomorphisme $\mho_3(F)$ dans la base monomiale de $\Jex_{1,d} \oplus \Smac_{0,d} \oplus \Jex_{1,d'}^\star$:
$$
    X^2Y,\ 
    XY^2,\ 
    XYZ,\ 
    XZ^2,\ 
    Y^3,\ 
    Y^2Z,\ 
    YZ^2,\ 
    Z^3\qquad | \qquad
    X^3,\ 
    X^2Z\qquad | \qquad
    xy,\ y^2,\ yz, \ z^2 
$$
Les 8 premiers monômes forment une base de $\Jex_{1,d}$ 
et les 4 derniers une base de $\Jex_{1,d'}^\star$. 
Et les 2 monômes au milieu ($X^3$ et $X^2Z$) indexent le bloc \MoutonNoir{}, ils forment une base de $\Smac_{0,d}$.

$$
\def\blz{\blacklozenge}
\def\lra{\leftrightarrow} \def\lrsa{\leftrightsquigarrow}
\mho_3(F) = 
\NorthEastBordermatrix{
\Veti{X^2Y\lra X^{2}\,e_{2}} &\Veti{XY^2\lra XY\,e_{2}} &\Veti{XYZ\lra XZ\,e_{2}} &\Veti{XZ^2\lra Z^2\,e_{3}}
&\Veti{Y^3\lra Y^{2}\,e_{2}} &\Veti{Y^2Z\lra YZ\,e_{2}} &\Veti{YZ^2\lra Z^{2}\,e_{2}} &\Veti{Z^3\lra Z\,e_{3}}
  &\Veti{xz \lrsa{X^3}} & \Veti{x^{2}\lrsa{X^2Z}} & \Veti{xy} & \Veti{y^{2}} & \Veti{yz} & \Veti{z^{2}} & \\
b_{2} & b_{1} & . & c_{2} & . & . & . & . &\sbullet & & & &\sbullet   &b & \Heti{X^{2}Y} \\
. & b_{2} & . & c_{4} & b_{1} & . & . & . &   &   &   &   &           &d & \Heti{XY^{2}} \\
. & b_{3} & b_{2} & c_{5} & . & b_{1} & . & c_{2} &   &   &   &   &    &e & \Heti{XYZ} \\
. & . & b_{3} & c_{6} & . & . & b_{1} & c_{3} &   &   &   &   &        &f & \Heti{XZ^{2}} \\
. & . & . & . & b_{2} & . & . & . &   &   &   &   &                   &g & \Heti{Y^{3}} \\
. & . & . & . & b_{3} & b_{2} & . & c_{4} &   &   &   &   &            &h & \Heti{Y^{2}Z} \\
. & . & . & . & . & b_{3} & b_{2} & c_{5} &   &   &   &   &            &i & \Heti{YZ^{2}} \\
. & . & . & . & . & . & b_{3} & c_{6} &   &   &   &   &                &j & \Heti{Z^{3}} \\
b_{1} & . & . & c_{1} & . & . & . & . &\blz & &   &   &                &a & \Heti{X^{3} \lrsa xz} \\
b_{3} & . & b_{1} & c_{3} & . & . & . & c_{1} & &\blz &   &   &\sbullet &c & \Heti{X^{2}Z \lrsa x^2} \\
. & . & . & . & . & . & . & . & b_{3} & b_{1} & b_{2} & . & . & . & \Heti{xy\leftrightarrow x\,e_{2}} \\
. & . & . & . & . & . & . & . & . & . & b_{1} & b_{2} & b_{3} & . & \Heti{y^2\leftrightarrow y\,e_{2}} \\
. & . & . & . & . & . & . & . & b_{1} & . & . & . & b_{2} & b_{3} & \Heti{yz \leftrightarrow z\,e_{2}} \\
. & . & . & . & . & . & . & . & c_{3} & c_{1} & c_{2} & c_{4} & c_{5} & c_{6} & \Heti{z^2 \leftrightarrow e_{3}} \\
}
$$
\noindent 
Expliquons les doubles flèches $\leftrightarrow$. 
En haut à gauche, elles désignent la correspondance $\Jex_{1,d} \leftrightarrow \Smac_{1,d}$, 
ce qui permet de remplir les colonnes en calculant 
directement $\Syl_d(\Smac_{1,d})$. 
En bas à droite, ces flèches désignent la 
correspondance $\Jex_{1,d'}^\star \leftrightarrow \Smac_{1,d'}^\star$ : 
on calcule $\Syl_{d'}(\Smac_{1,d'})$ et on dispose en ligne le résultat.
Quant au symbole~$\leftrightsquigarrow$, 
il désigne le mouton-swap qui indexe le bloc diagonal du milieu 
(qui est ici de taille 2 et dont la diagonale est matérialisée par des losanges).
Les deux blocs diagonaux extrêmes sont respectivement~$W_{1,d}$, de
taille 8 et la \emph {transposée} de $W_{1,d'}$, de taille 4 :
$$
W_{1,d} \ = \ 
\EastBordermatrix{
b_{2} & b_{1} & . & c_{2} & . & . & . & . & \Heti{X^{2}Y} \\ 
. & b_{2} & . & c_{4} & b_{1} & . & . & . & \Heti{XY^{2}} \\ 
. & b_{3} & b_{2} & c_{5} & . & b_{1} & . & c_{2} & \Heti{XYZ} \\ 
. & . & b_{3} & c_{6} & . & . & b_{1} & c_{3} & \Heti{XZ^{2}} \\ 
. & . & . & . & b_{2} & . & . & . & \Heti{Y^{3}} \\ 
. & . & . & . & b_{3} & b_{2} & . & c_{4} & \Heti{Y^{2}Z} \\ 
. & . & . & . & . & b_{3} & b_{2} & c_{5} & \Heti{YZ^{2}} \\ 
. & . & . & . & . & . & b_{3} & c_{6} & \Heti{Z^{3}} \\ 
}
\hspace{2cm}
W_{1,d'} \ = \ 
\EastBordermatrix{
b_{2} & b_{1} & . & c_{2} & \Heti{xy} \\ 
. & b_{2} & . & c_{4} & \Heti{y^{2}} \\ 
. & b_{3} & b_{2} & c_{5} & \Heti{yz} \\ 
. & . & b_{3} & c_{6} & \Heti{z^{2}} \\ 
}
$$
\end{exemple}

\subsubsection*{Mais où allons-nous avec $\mho_d$? Ou ce qui nous (vous?) attend}

Ici $D = (1,2,3)$ de degré critique $\delta=3$ avec:
$$
P_1 = a_1X_1 + a_2X_2 + a_3X_3,
\qquad\quad
P_2 = b_1X_1^2 + b_2X_1X_2 + b_3X_1X_3 + b_4X_2^2 + b_5X_2X_3 + b_6X_3^2
$$
Il est inutile de \og voir\fg{} $P_3$ car il n'est pas visualisé dans la matrice ci-dessous; mais bien sûr
il intervient (comme $P_1,P_2)$ dans les $\Bez_\ell$. 

\medskip
On fait le choix de $d=2$ donc $d' = 1$.
A quoi va servir le fourbi de $\mho_d$ ? Soit $\Bez(\uX,\uY)$ un bezoutien de $\uP$.
On écrit sa composante bihomogène $\Bez_{d,d'} \in \bfA[\uX,\uY]_{2,1}$ sous la forme:
$$
\Bez_{d,d'} = \Bez_1(\uX)Y_1 + \Bez_2(\uX)Y_2 + \Bez_3(\uX)Y_3, \qquad
\Bez_\ell(\uX) \in \bfA[\uX]_d
$$
Chaque $\Bez_\ell$ est un polynôme homogène de degré $d=2$ en $\uX$, et
chacun de ses coefficients est homogène de poids~1 en $P_i$.  On
reporte ces 3 polynômes $\Bez_\ell$ dans les dernières colonnes \og en $Y_\ell$ \fg{}. 

Concrètement, pour obtenir $\mho_{d}(\Bez_{d,d'})$, voici comment on
procède.  On calcule $\Syl_d(\Smac_{1,d})$ permettant d'obtenir les
premières colonnes, et la correspondance
$\Jex_{1,d} \leftrightarrow \Smac_{1,d}$ assure le bon appariement
ligne/colonne. On détermine ensuite $\Smac_{0,d}$ via sa base
monomiale. Ici c'est $(X_2X_3, X_3^2)$ dont l'image par le mouton-swap
est $(Y_3, Y_2)$. On détermine ensuite $\Syl_{d'}(\Smac_{1,d'})$
permettant de remplir le bloc en bas à droite, la correspondance
monomiale $\Smac_{1,d'} \leftrightarrow \Jex_{1,d'}$ assurant le fait
que colonnes et lignes de $\mho_d$ sont bien agencées.  On obtient
alors un endomorphisme de $\bfA[\uX]_d \oplus \Jex_{1,d'}^\star$,
schématisé à gauche ; à droite il s'agit de $\mho_d(F_{d,d'})$ où $F$ est le
bezoutien rigidifié du jeu étalon.
$$ 
\def\lra{\leftrightarrow} \def\lrsa{\leftrightsquigarrow}
\newcommand \veti[1] {\rotatebox{-90}{\mbox{$\scriptstyle#1$}}}
\mho_{d,\uP}(\Bez_{d,d'}) =
\NorthEastBordermatrix{
\Veti{X_1^2 \lra X_1\,e_1} & \Veti{X_1X_2 \lra X_2\,e_1} & \Veti{X_1X_3 \lra X_3\,e_1} & \Veti{X_2^2 \lra e_2} & \Veti{Y_3} & \Veti{Y_2} & \Veti{Y_1} & \\
a_1 & . & . & b_1     &              &              &               &\Heti{X_1^2} \\
a_2 & a_1 & . & b_2   &              &              &               &\Heti{X_1X_2} \\
a_3 & . & a_1 & b_3   &              &              &               &\Heti{X_1X_3} \\
. & a_2 & . & b_4     &\veti{\!\!\!\!\Bez_3} &\veti{\!\!\!\!\Bez_2} &\veti{\!\!\!\!\Bez_1}  &\Heti{X_2^2} \\
. & a_3 & a_2 & b_5   &              &              &               &\Heti{X_2X_3} \\
. & . & a_3 & b_6     &              &              &               &\Heti{X_3^2} \\
. & . & . & .         & a_3          & a_2          & a_1           &\Heti{Y_1 \leftrightarrow e_1} \\
}
\quad\quad
\EastBordermatrix{
a_1 & . & . & b_1     &.              &.              &.               &\Heti{X_1^2} \\
a_2 & a_1 & . & b_2   &.              &.              &.               &\Heti{X_1X_2} \\
a_3 & . & a_1 & b_3   &.              &.              &.               &\Heti{X_1X_3} \\
. & a_2 & . & b_4     &.              &.              &.          &\Heti{X_2^2} \\
. & a_3 & a_2 & b_5   &1              &.              &.               &\Heti{X_2X_3} \\
. & . & a_3 & b_6     &.              &1             &.               &\Heti{X_3^2} \\
. & . & . & .         & a_3          & a_2          & a_1           &\Heti{Y_1} \\
}
\qquad
$$
Notons:
$$
\calR = \calR(\uP) = \det\big(\mho_d(\Bez_{d,d'})\big)
$$
L'objectif va être de montrer que \boxed{\calR = \Res(\uP)}. Ce résultat sera obtenu en montrant que
$\calR \in \ElimIdeal$, en contrôlant son poids en chaque $P_i$ et en vérifiant que pour
le jeu étalon $\calR(\uX^D) = 1$. Ici on peut contrôler le poids en $P_i$ en développant
le déterminant par rapport à la dernière ligne:
$$
\poids_{P_1}(\calR) = 1+2+3 = 6, \qquad
\poids_{P_2}(\calR) = 0+2+1 = 3, \qquad
\poids_{P_3}(\calR) = 0+2+0 = 2
$$
On trouve $\poids_{P_i}(\calR) = \widehat{d}_i$! On comprendra plus tard que cela est dû à 
$d,d' < \min_{j\ne k}(d_j+d_k)=3$. La normalisation en le jeu étalon $\calR(\uX^D) = 1$ n'est pas compliquée
car non seulement le déterminant est 1 mais $\mho_{d,\uX^D}(\Bez_{d,d'}) = \Id$. En effet, de manière
générale, pour le jeu étalon
$$
\Bez^{\uX^D}(\uX,\uY) = \sum_\beta \Bez_\beta(\uX)Y^\beta  \qquad \text{avec} \qquad
\Bez_\beta = \begin {cases}
X^{\emouton-\beta} &\text{si $\beta \preccurlyeq \emouton$} \\
0                &\text{sinon} \\
\end {cases}
$$
Ce qui donne ici, dans l'ordre des colonnes, pour la spécialisation
de $\Bez_{d,d'}$ en le jeu étalon:
$$
\Bez_3 = X_2X_3, \qquad \Bez_2 = X_3^2, \qquad  \Bez_1 = 0
$$
Dans la matrice de droite, volontairement, on n'a pas spécialisé $\uP$ en
le jeu étalon mais seulement l'argument $\Bez_{d,d'}$, histoire de mieux
visualiser le bloc mouton-noir $\Smac_{0,d}$.
Il reste bien sûr un point important à montrer: $\calR \in \ElimIdeal$
(cela sera l'objet de la proposition~\ref{varpiElimIdeal}).
Voilà le programme qui nous attend et qui nous conduira en particulier à:
$$
\Res(\uP) = \dfrac{\det \mho_d(\Bez_{d,d'})}{\det W_{2,d}\det W_{2,d'}}
$$

\bigskip

L'exemple précédent illustre en particulier le comportement de $\mho_d$ pour
le jeu étalon. Nous formalisons le cas général dans la proposition suivante,
proposition dont la vérification est laissée au lecteur.

\begin{prop} [$\mho_d$ et le jeu étalon]
\label{MhoEtalon}
  
Le bezoutien rigidifié du jeu étalon $\uX^D$ est le le polynôme
$$
\Bez(\uX,\uY) = (XY)^\emouton \overset{\rm def}{=} \sum_{\alpha + \beta = \emouton} X^\alpha Y^\beta
$$
L'application linéaire $\Bez_{d,d'}^\dagger : \bfA[\uY]_{d'}^\star \to
\bfA[\uX]_d$ a pour restriction à $\Smac_{0,d'}^\star$ l'isomorphisme
mouton-swap $\Smac_{0,d'}^\star \overset{\simeq}{\longrightarrow}
\Smac_{0,d}$ et est nulle sur son supplémentaire monomial $\Jex_{1,d'}^\star$.

\medskip
En conséquence, pour le jeu étalon, on~a $\mho_d\big((XY)^\emouton_{d,d'}\big) = \Id$.
\end{prop}

\begin{prop}\label{ZigZagNul}
  
On rappelle la notation $s_d = \dim \Jex_{1,d}$ (rang de l'application de Sylvester $\Syl_d$).
La matrice de l'endomorphisme $\mho_d(F)$ est de taille $N = s_d + \chi_{d,d'} + s_{d'}$ et 
$$
N \ <\  
\dim \bfA[\uX]_d \,+\, \dim \bfA[\uY]_{d'}^\star
$$
En particulier, pour le polynôme nul $F = 0$, tout zig-zag de la matrice $\mho_d(0)$ est nul.
Note: nous désignons par zig-zag d'une matrice carrée $M$ de taille $N$, tout produit $m_{1,\tau(1)}
\cdots m_{N,\tau(N)}$ indexé par $\tau \in \fS_N$, terme qui 
apparaît dans le développement classique de $\det(M)$.
\end{prop}

\begin{proof} \leavevmode

L'espace sur lequel opère $\mho_d(F)$ est  
$$
\bfA[\uX]_d \oplus \Jex_{1,d'}^\star = \Jex_{1,d} \oplus \Smac_{0,d} \oplus  \Jex_{1,d'}^\star
\qquad \text{de dimension} \quad s_d + \chi_{d,d'} + s_{d'}
$$
D'où la formule donnée pour $N$.
Comme $d \leqslant \delta$, on a $s_d < \dim \bfA[\uX]_d$ d'où l'inégalité stricte annoncée pour $N$.

\medskip

Considérons le bloc rectangulaire \og en haut à droite\fg, indexé par $\bfA[\uY]_{d'}^\star 
\times \bfA[\uX]_d$, correspondant à $F^\dagger : \bfA[\uY]_{d'}^\star \to \bfA[\uX]_d$.
Si $F=0$, ce bloc est nul.
Notons $I$ l'ensemble des lignes indexées par $\bfA[\uX]_d$ et $J$ 
celui des colonnes indexées par $\bfA[\uY]_{d'}^\star$.
Comme  $N < \#I + \#J$, pour toute permutation $\tau \in \fS_N$, 
on ne peut pas avoir, pour des raisons de cardinaux,
$\tau(I) \subset \{1..N\} \setminus J$. 
Ainsi il existe $i \in I$ tel que $\tau(i) \in J$
si bien que le zig-zag de $\mho_d(0)$ indexé par $\tau$ est nul.

\end{proof}

\renewcommand \absM {3}    
\renewcommand \ordM {2}    
\renewcommand \taille {7}  

\subsection{L'application déterminant $\varpi_d = \varpi_{d,\protect\uP} :
\bfA[\protect\uX,\protect\uY]_{d,d'} \rightarrow \bfA$}

Comme auparavant, $d,d'$ sont deux entiers de somme le degré critique $\delta$.
Dans la section précédente est intervenue la construction:
$$
\mho_d = \mho_{d,\uP} : \bfA[\uX,\uY]_{d,d'} \ \longrightarrow \ 
\End\big(\bfA[\uX]_d \oplus \Jex_{1,d'}^\star\big)
$$
On peut alors définir $\varpi_d : \bfA[\uX,\uY]_{d,d'} \rightarrow \bfA$ 
par $\varpi_d(F) = \det\big(\mho_d(F)\big)$. 

\label{NOTA17-varpi}%

\begin{rmq}

Dans les trois pages suivantes, intervient le déterminant $\varpi_d(F)
= \det \mho_d(F)$ pour $F \in \bfA[\uX,\uY]_{d,d'}$.  Puisque $\det
\bbS_d(F) = \pm\varpi_d(F)$ avec un signe $\pm$ fixe indépendant de
$F$ (ce signe ne dépend que des bases monomiales dans lesquelles
$\bbS_d(F)$ est exprimé), on peut se permettre de remplacer
$\mho_d(F)$ par $\bbS_d(F)$. Ceci est sans influence sur la nature des
résultats comme le lecteur pourra le vérifier.  La considération de
$\bbS_d(F)$ présente un avantage : celui de pouvoir facilement
schématiser sa définition par une \og matrice \fg{}.

\end{rmq}

\subsubsection*{Un nouveau scalaire dans l'idéal d'élimination}

\begin{prop}\label{varpiElimIdeal}
Soit $\Bez = \Bez(\uX,\uY)$ un bezoutien du système $\uP$.  Alors
$\varpi_d(\Bez_{d,d'}) \in \ElimIdeal$.  Plus précisément, pour tout
$|\gamma| = \delta+1$, on a $X^\gamma\,\varpi_d(\Bez_{d,d'}) \in
\langle\uP\rangle_{\delta+1}$.
\end{prop}

Une remarque à propos de ``Plus précisément''.  Cela a du sens de se
placer en terrain générique, supposons-le un instant (cela ne sera pas
le cas dans la preuve) et notons $a$ le scalaire
$\varpi_d(\Bez_{d,d'})$.  Une fois acquis $a\in \ElimIdeal$, on a $H
a \in \uPsat_\delta$ pour tout $H \in \bfA[\uX]_\delta$ et, $\uP$
étant régulière, on sait alors depuis~\ref{MiniWiebe} que
$X_iHa \in \langle\uP\rangle_{\delta+1}$, rendant ainsi la précision
inutile.  Mais ici, dans la preuve, nous n'utilisons aucunement le
chapitre~\ref{ObjetsSuiteP}, seulement les propriétés spécifiques au
bezoutien $\Bez$ et à $\mho_d$.

\begin{proof}
Soit $X^\alpha$ avec $|\alpha| = d$. Nous allons montrer que
$X^\alpha\,\det\big(\bbS_d(\Bez_{d,d'})\big)$, égal à
$\pm X^\alpha \varpi_d(\Bez_{d,d'})$, 
appartient à $\uPsat_d$, ce
qui prouvera $\varpi_d(\Bez_{d,d'}) \in \ElimIdeal$.
Nous notons
$$
r = \dim\Jex_{1,d} = \dim \Smac_{1,d}, \qquad
r' = \dim\Jex_{1,d'} = \dim \Smac_{1,d'}
$$
Dans la matrice $\bbS_d(\Bez_{d,d'})$, à coefficients dans $\bfA$, 
chaque colonne $j$ parmi les $r$ premières est la vectorisation
d'un polynôme homogène $F_j$ de degré $d$, multiple d'un certain $P_i$,
donc dans $\langle\uP\rangle_d$.

\smallskip
\noindent
\parbox{0.65\linewidth}{ 

Dans cette matrice, on effectue l'opération élémentaire 
$$
\mathrm{Ligne}(\alpha) \ \leftarrow \ 
\sum_{|\alpha'| = d} X^{\alpha'} \mathrm{Ligne}(\alpha')
$$
Le déterminant de $\bbS_d(\Bez_{d,d'})$ est donc multiplié par $X^\alpha$ 
et égal au déterminant de la matrice ci-contre.
\`A présent, tout se passe dans cette nouvelle matrice dont la ligne $X^\alpha$ 
est à coefficients dans $\bfA[\uX]_d$ : les~$r$ premiers coefficients de cette ligne 
sont les polynômes $F_1, \dots, F_r$ et les coefficients suivants 
sont les $\Bez_\beta(\uX)$ pour $|\beta|=d'$ 
où $\Bez_{d,d'} = \sum_{|\beta| = d'}  \Bez_{\beta}(\uX) Y^\beta$.
}
\hspace{0.4cm}
\noindent
\parbox{0.33\linewidth}{
$$
\begin{tikzpicture}[scale = 0.6] 
\newcommand \ordL {3} 
\renewcommand \epaisseur {0.7}
\draw[fill=gray!15]  (0,\taille) rectangle \milieu ;
\path (0,\taille) -- \milieu node[midway] {\RestSyl{d}} ;
\draw[fill=gray!75]  (\taille,\taille) rectangle \milieu ;
\path (\taille,\taille) -- \milieu node[midway] {\hbox{$\Bez_{d,d'}^\dagger$}} ;
\draw[fill=gray!45]  (\taille, 0) rectangle \milieu ;
\path (\taille,0) -- \milieu node[midway] {\TransRestSyl{d'}} ;
\draw  (0,0) rectangle \milieu ;
\draw[fill=white] (0, \ordL) rectangle (\absM, \ordL + \epaisseur) ;
\path (0, \ordL) -- (\absM, \ordL + \epaisseur) node[midway] 
            {\hbox{\scriptsize $(F_j)_{1 \leqslant j \leqslant r}$}} ;
\draw[fill=gray!45]  (\absM, \ordL) rectangle (\taille, \ordL + \epaisseur) ;
\path (\absM, \ordL) -- (\taille, \ordL + \epaisseur) node[midway] 
            {\hbox{\scriptsize $(\Bez_\beta)_\beta$}} ;
\draw (\taille, \ordL + \epaisseur*0.5) node[right] {\scriptsize $X^\alpha$} ;
\end{tikzpicture}
$$
}

Utilisons un développement de Laplace en prenant pour ensemble $I_0$ de lignes 
la ligne $X^\alpha$ et les~$r'$ dernières lignes. On va montrer que tout
mineur $\Delta$ d'ordre $r'+1$ s'appuyant sur $I_0 \times J$ est dans $\uPsat_d$.
Il y a deux cas à distinguer en fonction de $J$.
Si $J$ s'appuie exclusivement sur les colonnes de droite, alors 
d'après~\ref{FormesInertieDegred}, le mineur $\Delta$ est dans $\uPsat_d$.
Si $J$ a une colonne qui \og mord \fg{} sur le bloc nul en bas à gauche, alors 
en développant $\Delta$ par rapport à cette colonne 
(qui est du type $\transpose (F_j, 0, \dots, 0)$ avec~$F_j$ dans~$\langle \uP \rangle_d$), 
on obtient $\Delta \in \langle\uP\rangle_d$, 
a fortiori dans $\uPsat_d$.

\medskip

Pour contrôler plus précisément l'appartenance de $\det\bbS_d(\Bez_{d,d'})$ 
à l'idéal d'élimination $\ElimIdeal$, il 
suffit de multiplier la ligne $X^\alpha$ par un monôme quelconque $X^{\gamma'}$ de degré $d'+1$
et d'utiliser~\ref{FormesInertieDegred}.
\end{proof}

\begin{rmq}
En degré $\delta$, on a $\omega(F) G - \omega(G) F \in \uPdelta$, si bien 
que $\omega(F)$ appartient à $\ElimIdeal$ dès que $F$ appartient à $\uPdelta^\sat$.
Il en va de même dans ce contexte $(d,d')$ ; 
autrement dit, $\varpi_d(F) \in \ElimIdeal$ dès que 
$F \in \langle \uPX, \uPY \rangle^\sat_{d,d'}$.
Pour s'en convaincre, reprendre 
la preuve précédente et constater que la seule propriété utilisée de $\Bez$
est $\Bez_{d,d'} \in \langle \uPX, \uPY \rangle^\sat_{d,d'}$.
\end{rmq}

\subsubsection*{Pseudo-linéarité de $\varpi_d$}

En degré $\delta$, on a $\omega(F + H) = \omega(F)$ pour $H \in
\uPdelta$ ; cette égalité est immédiate une fois que l'on sait que
$\omega$ est une forme linéaire nulle sur $\uPdelta$.  On a le même
type d'égalité pour $\varpi_d$.  Cependant, on ne peut pas utiliser le
même argument, car $\varpi_d$ n'est pas linéaire !  On s'en sort quand
même grâce à de petites perturbations élémentaires.  Le résultat
repose alors sur la multilinéarité-alternée du déterminant et sur le
rang de l'application de Sylvester $\Syl_d = \Syl_d(\uP)$.
Concernant le rang, faisons la remarque banale suivante:
soit $\uu$ une famille finie de vecteurs de $\bfA^m$ de rang $< p$
au sens habituel de la nullité de l'idéal déterminantiel $\calD_p(\uu)$.
Alors, pour des vecteurs quelconques $v_1, \dots, v_q \in \bfA^m$,
la famille $(\uu, v_1, \dots, v_q)$ est de rang $<p+q$, comme on le
voit facilement par récurrence sur $q$.
Bien entendu, l'algèbre extérieure trivialise
cette observation car elle permet d'écrire:
$$
u_{i_1}\wedge\cdots\wedge u_{i_p} = 0   \quad\Longrightarrow\quad
u_{i_1}\wedge\cdots\wedge u_{i_p} \wedge v_1 \wedge \cdots\wedge v_q = 0
$$
Dans la preuve du lemme ci-dessous, en notant $r=\dim\Jex_{1,d}$ le
rang de $\Syl_d$, on appliquera cette remarque à $r+1$ vecteurs de $\Im\Syl_d =
\langle\uP\rangle_d$. Idem avec~$\Syl_{d'}$ en notant
$r'=\dim\Jex_{1,d'}$ son rang.

\begin{lem} 
Soit $F \in \bfA[\uX,\uY]_{d,d'}$.
\begin{enumerate}[\rm i)]
\item 
Pour $|\beta| = d'$ et $H_\beta(\uX) \in \langle \uPX \rangle_d$, on a 
$\varpi_d\big(F + H_\beta(\uX) Y^\beta\big) \, = \, \varpi_d(F)$.
\item 
Pour $|\alpha| = d$ et $H_\alpha(\uY) \in \langle \uPY \rangle_{d'}$, on a 
$\varpi_d\big(F + X^\alpha H_\alpha(\uY)\big) \, = \, \varpi_d(F)$.
\item
Pour $H \in \big\langle \uPX, \uPY \big\rangle_{d,d'}$, 
on a $\varpi_d(F+H) \, = \, \varpi_d(F)$.
\end{enumerate}
\end{lem}

\begin{proof}\leavevmode

\noindent
i) Soit $H = H_\beta(\uX)Y^\beta = \big(\sum_{|\alpha|=d} a_\alpha X^\alpha \big) Y^\beta$.
Examinons la colonne indexée par $Y^\beta$ de $\bbS_d(F+H)$, 
les autres colonnes étant identiques à celles de $\bbS_d(F)$.
$$
\mathrm{Col}_\beta\big(\bbS_d(F+H)\big)
\ = \ 
\mathrm{Col}_\beta\big(\bbS_d(F)\big) 
\ + \ 
\EastBordermatrix{
& \\ 
                   &\bfA[\uX]_d \\
(a_{\alpha})_{\alpha} \\
\hligne{0.6pt}\\
0                  &\Smac_{1,d'}^\star
}
$$
\parbox{0.7\linewidth}{
Par multilinéarité du déterminant, 
$\varpi_d(F+H) = \varpi_d(F) + \det\big(\bbS^\beta_d(F)\big)$, 
où $\bbS^\beta_d(F)$ est la matrice obtenue à partir de $\bbS_d(F)$ en remplaçant sa 
colonne~$Y^\beta$ par la colonne ci-dessus. Pour conclure, il suffit de montrer que 
son déterminant est nul.
Or cette matrice $\bbS^\beta_d(F)$ est du type de la matrice ci-contre.
Par hypothèse, le vecteur colonne $(a_\alpha)_\alpha$ 
est combinaison linéaire des colonnes de $\Syl_d$.
Donc la colonne $Y^\beta$ se retrouve être du même type que les $r$ premières colonnes 
(d'où les couleurs !).
Maintenant, on peut constater 
que les mineurs d'ordre $r+1$ extraits sur les colonnes \og gris-blanc \fg{} 
sont nuls (ceci est dû au fait que les mineurs d'ordre $r+1$ de~$\Syl_d$ sont nuls, 
et à la présence des zéros dans le bas de la colonne~$Y^\beta$).
En utilisant la remarque précédant le lemme, on obtient
$\det\big(\bbS^\beta_d(F)\big) = 0$.
}
\hspace{0.5cm}
\parbox{0.28\linewidth}{ 
\begin{tikzpicture}[scale = 0.5]
\newcommand \absC {3.9}  
\draw[fill=gray!15]  (0,\taille) rectangle \milieu ;
\path (0,\taille) -- \milieu node[midway] 
        {${r \, \mathrm{col.} \atop \mathrm{de}\,\Syl_d}$} ;
\draw[fill=gray!75]  (\taille,\taille) rectangle \milieu ;
\path (\taille,\taille) -- \milieu node[midway] {\hbox{$F^\dagger$}} ;
\draw[fill=gray!45]  (\taille, 0) rectangle \milieu ;
\path (0, \ordM*0.5) -- (\taille, \ordM*0.5)  node[near end] 
        {${r' \, \mathrm{lig.} \atop \mathrm{de}\, \transpose {\Syl_{\!d'}} }$} ;
\draw  (0,0) rectangle \milieu ;
\draw[fill=gray!15]  (\absC, \taille) rectangle (\absC + \epaisseur, \ordM) ;
\draw[fill=white] (\absC, \ordM) rectangle (\absC + \epaisseur, 0) ;
\draw (\absC + \epaisseur, \taille) node[above] {\scriptsize $Y^\beta$} ;
\end{tikzpicture}
}

\medskip
\noindent
\parbox{0.7\linewidth}{
ii) 
Soit $H = X^\alpha H_\alpha(\uY) = X^\alpha \big( \sum_{|\beta| = d'} a_\beta Y^\beta \big)$. 
La démonstration est identique au point i), à transposée près. 
Toutes les lignes de $\bbS_d(F+H)$ sont identiques à celles de $\bbS_d(F)$, 
sauf la ligne indexée par~$X^\alpha$ : 
$$
\mathrm{Ligne}_\alpha\big(\bbS_d(F+H)\big) \ = \ 
\mathrm{Ligne}_\alpha\big(\bbS_d(F)\big) \ + \ 
\NorthEastBordermatrix{
\ \Smac_{1,d} &&  \bfA[\uY]_{d'}^\star &  \\
 0  &  \VR  (a_\beta)_{\beta} &  \\
}
$$
Pour conclure, il faut et il suffit de montrer que le déterminant de la matrice ci-contre 
est nul. La ligne \og blanc-gris \fg{} indexée par $X^\alpha$ est le vecteur ligne 
à droite dans l'égalité.
Tous les mineurs d'ordre $r'+1$ qui s'appuient sur les 
lignes blanc-gris sont nuls: en effet, par hypothèse 
$H_\alpha(\uY) \in \langle \uPY \rangle_{d'}$, donc 
le vecteur ligne $(a_\beta)_\beta$
est combinaison linéaire des lignes de $\transpose \Syl_{d'}$ ; de plus, 
$\Syl_{d'}$ est de rang~$r'$.
Finalement, le déterminant de cette matrice est bien nul.
}
\hspace{0.5cm}
\parbox{0.28\linewidth}{ 
\begin{tikzpicture}[scale = 0.5]
\newcommand \ordL {5.2} 
\draw[fill=gray!15]  (0,\taille) rectangle \milieu ;
\path (0,\taille) -- \milieu node[midway] 
        {${r \, \mathrm{col.} \atop \mathrm{de}\,\Syl_d}$} ;
\draw[fill=gray!75]  (\taille,\taille) rectangle \milieu ;
\path (\taille,\taille) -- \milieu node[midway] {\hbox{$F^\dagger$}} ;
\draw[fill=gray!45]  (\taille, 0) rectangle \milieu ;
\path (\taille,0) -- \milieu node[midway] 
        {${r' \, \mathrm{lig.} \atop \mathrm{de}\, \transpose {\Syl_{\!d'}} }$} ;
\draw  (0,0) rectangle \milieu ;
\draw[fill=white] (0, \ordL) rectangle (\absM, \ordL + \epaisseur) ;
\draw[fill=gray!45]  (\absM, \ordL) rectangle (\taille, \ordL + \epaisseur) ;
\draw (\taille, \ordL + \epaisseur*0.5) node[right] {\scriptsize $X^\alpha$} ;
\end{tikzpicture}
}

\medskip
\noindent
iii) 
En itérant l'égalité i), on obtient $\varpi_d(F+H^{(1)}) = \varpi_d(F)$ 
pour tout $H^{(1)} \in \langle \uPX \rangle_d \, \bfA[\uY]_{d'}$. 
De la même manière, l'égalité ii) fournit 
$\varpi_d(F+H^{(2)}) = \varpi_d(F)$ pour tout $H^{(2)} \in \langle \uPY \rangle_{d'} \, \bfA[\uX]_{d}$. 
Or un polynôme $H \in \big\langle \uPX, \uPY \big\rangle_{d,d'}$ 
s'écrit comme la somme $H^{(1)} + H^{(2)}$, avec des polynômes comme ci-dessus.
On a donc 
$$
\varpi_d(F + H) \ =\  
\varpi_d(F + H^{(1)} + H^{(2)}) \ =\  
\varpi_d(F + H^{(1)}) \ =\ 
 \varpi_d(F)
$$
\end{proof}

\begin{prop} \label{PseudoLineariteVarpi}
L'application $\varpi_d$ passe au quotient modulo 
$\big\langle \uPX, \uPY \big\rangle_{d,d'}$ ; autrement dit, 
pour $F, G \in \bfA[\uX,\uY]_{d,d'}$, on a 
$$
F \equiv G \, \bmod \big\langle\uPX,\uPY\big\rangle_{d,d'} 
\qquad \Longrightarrow \qquad 
\varpi_d(F) = \varpi_d(G)
$$
De plus, $\varpi_d$ est nulle sur $\big\langle \uPX, \uPY \big\rangle_{d,d'}$.
\end{prop}

\begin{proof}
C'est le contenu du point iii) du lemme précédent.
La nullité de $\varpi_d$ sur $\big\langle \uPX, \uPY \big\rangle_{d,d'}$ 
est due au fait que $\varpi_d(0) = 0$, cf.~\ref{ZigZagNul}.
\end{proof}

\subsubsection*{Indépendance du choix de la matrice bezoutienne}

La propriété d'invariance du choix de la matrice bezoutienne n'est
nullement une évidence et n'est pas abordée dans l'article~\cite{AD}
où les auteurs utilisent uniquement le bezoutien \og rigidifié\fg.  Elle n'est
pas prouvée non plus dans~\cite{J7} et pourtant, Jean-Pierre Jouanolou
a réalisé un travail important dans ce domaine avec l'utilisation de
techniques avancées et d'un langage sophistiqué.  Il affirme, au début
de~{3.11.19.11}~: \og Faute de pouvoir maîtriser la dépendance de
l'élément $L_\nu([h_{ij}]; \varphi_1,\dots,\varphi_n)$ de la décomposition
$[h_{ij}]$ choisie, nous allons nous limiter aux décompositions qui proviennent
\dots \fg{}.  On peut même penser que cette propriété
d'indépendance a été l'une de ses préoccupations puisque l'on constate
qu'il a étudié plusieurs constructions de matrices bezoutiennes dans
la section~{3.11}, confer les exemples~{3.11.7}, la
sous-section~{3.11.18}, ainsi que les premiers paragraphes
de~{3.11.19.11.}

Quant à nous, avouons notre surprise : la propriété
d'indépendance du choix de la matrice bezoutienne n'est pratiquement
pas liée aux matrices bezoutiennes ! 
Elle repose essentiellement sur la proposition précédente.

\begin{theo}[Indépendance du choix de la matrice bezoutienne] \label{IndependanceMatBez}
Pour deux bezoutiens $\Bez^{(1)}$,  $\Bez^{(2)}$ d'un même système $\uP$, 
on~a $\varpi_d(\Bez_{d,d'}^{(1)}) \, = \, \varpi_d(\Bez_{d,d'}^{(2)})$.
\end{theo}

\begin{proof}
D'après~\ref{ProprietesMorley}, 
on a $\Bez_{d,d'}^{(1)} \equiv \Bez_{d,d'}^{(2)} 
\bmod \big\langle \uPX, \uPY \big\rangle_{d,d'}$.
On conclut avec~\ref{PseudoLineariteVarpi}.
\end{proof}

\subsubsection*{Structure de $\varpi_d$, comme somme de produit de trois mineurs signés}

\begin{prop} \label{varpiStructure}
L'endomorphisme $\mho_d(F)$ opère sur le module $\Jex_{1,d} \oplus \Smac_{0,d} \oplus \Jex_{1,d'}^\star$
de dimension $s_d + \chi_d + s_{d'}$.  
Son déterminant $\varpi_d(F)$ est une somme signée de termes du type :
$$
\text{mineur d'ordre $s_d$ de $\Syl_d(\uP)_{|\Smac_{1,d}}$} 
\ \times \ 
\text{mineur d'ordre $\chi_d$ de $F^{\dagger}$}
\ \times \ 
\text{mineur d'ordre $s_{d'}$ de $\transpose \Syl_{d'}(\uP)_{|\Smac_{1,d'}}$}
$$
\end{prop}

\begin{proof}
Commençons par énoncer et démontrer le lemme suivant:

\parbox{0.6\linewidth}{
\itshape
Considérons une matrice carrée $A$ découpée de la façon suivante, 
avec $\#J_0=r$ et $\#I_0 = r'$. 

On a alors 
$$
\det A \ =\ 
\sum_{\#I=r \atop I \subset \overline {I_0}} 
\sum_{\#J = r' \atop J \subset \overline{J_0}}
\pm \  
\det A_{I \times J_0} \ 
\det A_{\overline {I_0} \setminus I \,\times\, \overline {J_0} \setminus J} \ 
\det A_{I_0 \times J} \ 
$$
avec $\pm \, = \, 
\varepsilon(J_0, \overline{J_0}) \ 
\varepsilon(I, \overline{I})\ 
\varepsilon(I_0, \,\overline {I_0} \setminus I) \ 
\varepsilon(J, \, \overline {J_0} \setminus J)
$.
}
\parbox{0.3\linewidth}{
\begin{tikzpicture}[scale = 0.5]
\clip (-1,0) rectangle (8.5,8.5) ;
\draw[fill=gray!20]  (0,\taille) rectangle \milieu ;
\draw[fill=gray!60]  (\taille,\taille) rectangle \milieu ;
\path (0,0) -- \milieu node[midway] {\scriptsize bloc nul} ;
\draw[fill=gray!20]  (\taille, 0) rectangle \milieu ;
\draw  (0,0) rectangle \milieu ;
\path \NorthUn node[midway,above] {$J_0$} ;
\path \EastDeux node[midway,right] {$I_0$} ;
\end{tikzpicture}
}

\medskip
\noindent
La formule du développement de Laplace suivant les colonnes de $J_0$ fournit :
$$
\det A  \ = \ 
\varepsilon(J_0, \overline{J_0}) 
\sum_{\#I=r} 
\varepsilon(I, \overline I) \ 
\det A_{I \times J_0} \ 
\det A_{\overline I \times \overline {J_0}} \ 
$$
On peut se restreindre aux $I$ tels que $I \subset \overline{I_0}$,
sinon la matrice $A_{I \times J_0}$ possède une ligne de zéros et
fournit une contribution nulle à la somme.  Développons maintenant
avec Laplace le déterminant de chaque matrice carrée $B :=
A_{\overline I \times \overline {J_0}}$ relativement à l'ensemble de
lignes $I_0 \subset \overline I$; par conséquent, la sommation porte
sur les $J$ tels que $J \subset \overline{J_0}$ et $\#J = r'$.
Notons $\widetilde I_0 = \overline I \setminus I_0$ le complémentaire de $I_0$ dans $\overline I$, 
et de manière analogue $\widetilde J = \overline {J_0} \setminus J$. Alors:
$$
\det A_{\overline I \times \overline {J_0}} \overset{\rm def.}{=}
\det B  \ = \ 
\varepsilon(I_0, \widetilde{I_0}) 
\sum_{\#J=r' \atop J \subset \overline{J_0}} 
\varepsilon(J, \widetilde J) \
\det B_{\widetilde I_0 \times \widetilde J} \ 
\det B_{I_0 \times J}
$$
On reporte cette valeur dans le développement de $\det A$ en utilisant
$\widetilde {I_0} \overset {\rm def.}{=} \overline I\setminus I_0 = \overline {I_0}\setminus I$,
ce qui termine la preuve du lemme.

\medskip
Revenons au contexte de l'énoncé en notant $J_0$ les colonnes indexées
par $\Jex_{1,d}$ issues de $\Syl_d(\uP)$, et $I_0$ les lignes indexées
par~$\Jex_{1,d'}$ issues de~${\transpose \Syl_{d'}(\uP)}$.
D'après le lemme, on voit que $\varpi_d(F)$ est une somme sur $I$ et $J$ de termes du type :
$$
\pm \quad  
\det \mho_d(F)_{I \times J_0} \ \times \ 
\det \mho_d(F)_{\overline{I_0} \setminus I \,\times\, \overline{J_0} \setminus J} 
\ \times \ 
\det \mho_d(F)_{I_0 \times J}
$$
avec $I \subset \overline{I_0}$ de cardinal $s_d$ et $J \subset
\overline{J_0}$ de cardinal $s_{d'}$.  Les mineurs figurant aux
termes extrêmes ne dépendent que de la suite $\uP$ et pas de $F$,
tandis que le mineur du milieu ne dépend lui que de $F$.
\end{proof}

\begin{prop}[Normalisation] \label{varpiEtalon}
Soit $\Bez = \Bez(\uX,\uY)$ un bezoutien quelconque du jeu étalon~$\uX^D$.
Alors $\varpi_{d,\uX^D}(\Bez_{d,d'}) = 1$.
\end{prop}

\begin{proof}

En un premier temps, considérons la matrice bezoutienne \og
rigidifiée\fg{} du jeu étalon figurant en fin de remarque~\ref{BezoutianRemarks}
dont le déterminant vaut $(XY)^\emouton = \sum_{\alpha + \beta =
  \emouton} X^\alpha Y^\beta$. On a vu en~\ref{MhoEtalon} que
$\mho_{d,\uX^D}\big((XY)^\emouton_{d,d'}\big) = \Id$, a fortiori de
déterminant~1.  Pour un bezoutien quelconque, on invoque le théorème
d'indépendance (théorème~\ref{IndependanceMatBez}).
\end{proof}

\begin{prop}\label{varpiEvalEtalon}
On a $\varpi_d\big((XY)^\emouton_{d,d'}\big) = \det W_{1,d} \times \det W_{1,d'}$. 
\end{prop}

\begin{proof}

Posons $F = (XY)^\emouton_{d,d'}$ de sorte que $F^\dagger : \bfA[\uY]_{d'}^\star \to \bfA[\uX]_d$
est l'isomorphisme mouton-swap sur~$\Smac_{0,d'}^\star$ et $0$ sur $\Jex_{1,d'}^\star$.
Reprenons les notations de la formule figurant à la fin de la définition dite géante~\ref{DefGeante}:
$$ 
(X^\alpha\oplus X^\gamma)\oplus (Y^\beta)^\star \in
\bfA[\uX]_d\oplus\Jex_{1,d'}^\star = (\Jex_{1,d}\oplus\Smac_{0,d}) \oplus\Jex_{1,d'}^\star,
\qquad
Z = (Y^{\emouton-\gamma})^\star + (Y^\beta)^\star \in \bfA[\uY]_{d'}^\star
$$
Alors $F^\dagger(Z) = X^\gamma$ et $\mho_d(F)$ réalise
$$
(X^\alpha \oplus X^\gamma) \oplus (Y^\beta)^\star \mapsto
\big((\Syl_d \circ \varphi)(X^\alpha) + X^\gamma\big) \ \oplus \ 
\big(\transpose \varphi \circ \transpose \Syl_{d'}\big)(Z)
$$ 
Sa matrice, relativement à la somme directe des 3 facteurs, est triangulaire supérieure en 3 blocs:
$$
\mho_d(F) : 
\begin {bmatrix}
W_{1,d} &0              &0 \\
\star  &\Id_{\Smac_{0,d}} &0 \\  
0      &\star          &\transpose{W_{1,d'}} \\
\end {bmatrix}
$$
D'où la formule annoncée pour le déterminant.
\end{proof}

\begin{prop}[\'Epuration] \label{varpiDivisibilite}
Le déterminant $\varpi_d(F)$ est divisible par
$\Delta_{2,d} \times \Delta_{2,d'}$, c'est-à-dire
par $\det W_{2,d} \times \det W_{2,d'}$.
\end{prop}

\begin{proof}
D'après \ref{varpiStructure}, $\varpi_d(F)$ est une somme signée de termes du type 
$$
\text{mineur d'ordre $s_d$ de $\Syl_d(\uP)_{|\Smac_{1,d}}$} 
\ \times \ 
\text{mineur d'ordre $\chi_d$ de $F^{\dagger}$}
\ \times \ 
\text{mineur d'ordre $s_{d'}$ de $\transpose \Syl_{d'}(\uP)_{|\Smac_{1,d'}}$}
$$
Or, par définition ou presque, le scalaire $\Delta_{2,d}$ divise tous les mineurs d'ordre $s_d$ de la restriction
$(\Syl_d)_{|\Smac_{1,d}}$, cf la proposition \ref{DeltakdDiviseMineurs}
et le point iv) de la proposition~\ref{SuiteDeltak} qui en fournit un certificat:
$$
\calD_{s_d}\big((\Syl_d)_{|{\Smac_{1,d}}}\big) = \Delta_{2,d}\,\rmc(\Theta_{1,d})
$$
De plus $\Delta_{2,d} = \det W_{2,d}$ (cf théorème \ref{DeltakdProdBinomialDetWhd}).
Idem avec $d'$, d'où le résultat annoncé.
\end{proof}

\subsection{Symétrie quand tu me tiens}

Soit $\dsV(\uX,\uY)$ une matrice bezoutienne de $\uP$, de déterminant
$\Bez(\uX,\uY)$; on se gardera de croire que $\Bez(\uX,\uY)$ est symétrique en $(\uX,\uY)$. 
En voici un exemple ultra-simple, de format $D = (1,2)$ :
$$
P_1 = aX_1 \qquad P_2 = bX_1X_2
$$
La matrice bezoutienne, celle spécifiée dans la définition~\ref{DefBezoutien}, se détermine
aisément ainsi que son déterminant :
$$
\dsV(\uX,\uY) = \begin {bmatrix} a & bX_2\\ 0 & bY_1\\ \end {bmatrix}, \qquad
\Bez(\uX,\uY) = abY_1
$$
Il est clair que le bezoutien n'est pas symétrique en ($\uX, \uY$).
Cette propriété (erronée) de symétrie du bezoutien est cependant
affirmée et utilisée dans l'article~\cite{AD} à la première ligne de la preuve du lemma~2.1.
Les auteurs vont alors déduire des égalités
matricielles, qui se révèlent fausses. 
Mettons par exemple en défaut leur remark 2.6 dans laquelle est mentionné
le fait que deux matrices sont transposées l'une de l'autre modulo un mystérieux
\og with a careful inspection, we have that ordering properly their rows and 
columns\dots\fg{}. 
Choisissons le format $D = (2,3)$ pour lequel $\delta = 3$ :
$$
\setlength{\tabcolsep}{2pt}
\left\{
\begin{tabular}{rcp{12cm}} 
$P_{1}$ & $=$ & $a_{1}X_{1}^{2} + a_{2}X_{1}X_{2} + a_{3}X_{2}^{2}$\\ [0.1cm] 
$P_{2}$ & $=$ & $b_{1}X_{1}^{3} + b_{2}X_{1}^{2}X_{2} + b_{3}X_{1}X_{2}^{2} + b_{4}X_{2}^{3}$\\ [0.1cm] 
\end{tabular}
\right.
$$
Prenons \boxed{(d,d') = (0,\delta) = (0,3)} ; les matrices en question, qui devraient
être transposées l'une de l'autre, \emph{à permutation près des lignes et colonnes},
sont $\mho_\delta(\Bez_{\delta,0})$ et $\mho_0(\Bez_{0,\delta})$. 
Ce qui n'est pas possible puisque l'on~a~:
{\small
$$
\hbox{\normalsize $\mho_\delta(\Bez_{\delta,0})$} =
\EastBordermatrix{
a_{1}&  . & b_{1}&  .  & \Heti{X_1^3} \\ 
a_{2}& a_{1}& b_{2}& -a_{3}b_{1} & \Heti{X_1^2X_2} \\ 
 . & a_{3}& b_{4}& a_{2}b_{4} -a_{3}b_{3} & \Heti{X_2^3} \\ 
a_{3}& a_{2}& b_{3}& a_{1}b_{4} -a_{3}b_{2} & \Heti{X_1X_2^2}
}
\ 
\hbox{\normalsize $\mho_0(\Bez_{0,\delta})$} = 
\EastBordermatrix{
a_1b_4 & a_{1}b_{2} -a_{2}b_{1}  & a_{1}b_{3} -a_{3}b_{1} &. & \Heti{1}\\ 
a_3 & a_{1}& a_2 & . & \Heti{(X_1^3)^\star} \\ 
a_2 & . & a_{1}& a_3 & \Heti{(X_1^2X_2)^\star}\\ 
b_3 & b_1 & b_2 & b_{4} & \Heti{(X_2^3)^\star} \\ 
}
$$
}

\medskip
\noindent
De notre côté, nous n'utilisons pas bien sûr la prétendue symétrie des
bezoutiens, et nous n'avons pratiquement jamais besoin de telles
comparaisons dans lesquelles interviendraient d'une part~$(\uX,\uY)$ 
et d'autre part~$(\uY,\uX)$. 
Cependant, suite à l'article~\cite{AD}, nous avons enquêté 
sur cette histoire. Nous en avons déduit une belle égalité de
symétrie (cf.~\ref{SymetrieVarpi}),  
qui affirme en quelque sorte qu'il y a moitié moins de 
formes d'inertie de degré $0$ du type $\varpi_d(\Bez_{d,d'})$ 
que de couples $(d,d')$ tels que $d+d'=\delta$, 
propriété nullement visible dans~\ref{varpiElimIdeal}.
Essayons de décrire, à l'aide de l'exemple, les symétries et non-symétries, 
et par la même occasion, l'architecture de la preuve de ce résultat.
Il est important de noter que dès le départ, 
nous n'utilisons pas de matrice bezoutienne spécifique, 
même si nous avons un petit faible pour celle spécifiée en~\ref{DefBezoutien} car
c'est en partie sur ce choix que repose notre implémentation et les
exemples qui en sont tirés (notre implémentation a cependant la
possiblité d'utiliser n'importe quelle matrice bezoutienne). 
Pour n'importe quelle matrice bezoutienne $\dsV(\uX,\uY)$ associée à $\uP$, 
il est facile de constater que la
matrice $\dsV^*(\uX,\uY) \overset {\rm def}{=} \dsV(\uY,\uX)$ est
également une matrice bezoutienne, de déterminant $\Bez^*(\uX,\uY) = \Bez(\uY,\uX)$. 
La prétendue symétrie se redresse en une vraie congruence:
$$
\Bez(\uX,\uY) \equiv \Bez^*(\uX,\uY) \ \bmod\ \langle\uP(\uX) - \uP(\uY)\rangle
$$
Cependant, malgré cette symétrie dans le quotient,
les \emph{endomorphismes} $\mho_{d'}(\Bez_{d',d})$ et $\mho_{d'}(\Bez^*_{d',d})$ ne
sont pas du tout de même nature.
Dans l'exemple, il est facile de constater que $\mho_\delta(\Bez_{\delta,0})$ et $\mho_\delta(\Bez^*_{\delta,0})$ 
(voir plus bas la représentation matricielle de ce dernier)
n'ont quasiment rien en commun (ils n'ont pas même trace). 
Mais leurs déterminants sont égaux (cf.~\ref{IndependanceMatBez}) 
et ceci n'a rien à voir avec la permutation $(\uX,\uY)$ !
En revanche, les deux endomorphismes $\mho_{d'}(\Bez^*_{d',d})$ et $\transpose{\mho_d(\Bez_{d,d'})}$
sont fortement liés puisqu'ils sont conjugués de manière explicite par 
une matrice de permutation.
Dans l'exemple, il s'agit des endomorphismes fournis sous forme matricielle :
$$
\mho_\delta(\Bez^*_{\delta,0}) =
\EastBordermatrix{
a_{1}&  . & b_{1}&  a_{1}b_{2} -a_{2}b_{1}  & \Heti{X_1^3} \\ 
a_{2}& a_{1}& b_{2}& a_{1}b_{3} -a_{3}b_{1} & \Heti{X_1^2X_2} \\ 
 . & a_{3}& b_{4}& . & \Heti{X_2^3} \\ 
a_{3}& a_{2}& b_{3}& a_1b_4 & \Heti{X_1X_2^2}
}
\qquad
\transpose{\,\mho_0(\Bez_{0,\delta})} =
\EastBordermatrix{
a_1b_4 & a_3 & a_2 & b_3 & \Heti{1^\star} \\
a_1b_2 - a_2b_1 & a_{1} &  . & b_{1} & \Heti{X_1^3} \\
a_1b_3 - a_3b_1 & a_{2} &  a_{1} & b_{2} & \Heti{X_1^2X_2} \\
. & . & a_3 & b_4 & \Heti{X_2^3} 
}
$$
Le résultat précis qui figure ci-dessous en~\ref{PermutationXYdansMho} ainsi que sa preuve, 
nous permettent d'affirmer qu'une conjugaison possible est :
$$
\mho_\delta(\Bez^*_{\delta,0}) = P_\tau^{-1}\, \transpose{\,\mho_0(\Bez_{0,\delta})}\, P_\tau
\qquad 
\text{avec } \quad
\tau = (1,2,3,4),
\qquad
P_\tau =
\left[
\begin{array}{*{4}{c}}
 . &  . &  . & 1 \\ 
1&  . &  . &  .  \\ 
 . & 1&  . &  .  \\ 
 . &  . & 1&  .  \\ 
\end{array}
\right]
$$
Insistons sur le fait que rien n'est matriciel, tout est endomorphisme. Ainsi la
matrice de permutation~$P_\tau$ ci-dessus s'obtient via un isomorphisme monomial :
$$
\Jex_{1,d'} \oplus \Smac_{0,d'} \oplus \Jex_{1,d}^\star 
\quad \longrightarrow \quad 
\Jex_{1,d}^\star \oplus \Smac_{0,d}^\star \oplus \Jex_{1,d'}
$$
où ici (en rappelant que $D = (2,3)$, $d=0$, $d'=3$):
$$
\Jex_{1,d'} = \bfA X_1^3 \oplus \bfA X_1^2 X_2 \oplus \bfA X_2^3, \qquad
\Smac_{0,d'} = \bfA X_1X_2^2, \qquad
\Jex_{1,d} = 0, \qquad
\Smac_{0,d} = \bfA.1
$$
Concluons. Indépendamment de la permutation $(\uX,\uY)$, on obtient 
$\varpi_\delta(\Bez_{\delta,0}) = \varpi_\delta(\Bez^*_{\delta,0})$ ;
puis, à l'aide de la conjugaison, $\varpi_\delta(\Bez^*_{\delta,0}) = \varpi_0(\Bez_{0,\delta})$.
Finalement, on récupère la jolie égalité 
${\varpi_\delta(\Bez_{\delta,0}) = \varpi_0(\Bez_{0,\delta})}$. 
La suite est donc consacrée à prouver cette formule avec $(d,d')$ quelconque
où bien sûr $d+d' = \delta$.

\subsubsection*{Les choses de la vie: $(\uX,\uY)$ versus $(\uY,\uX)$}

Dans la définition~\ref{DefGeante}, étant donné $F
= \sum_{|\alpha|=d\atop|\beta|=d'} a_{\alpha\beta}X^\alpha Y^\beta
\in \bfA[\uX,\uY]_{d,d'}$, on a fait un choix pour définir l'application linéaire $F^\dagger
: \bfA[\uY]_{d'}^\star \to \bfA[\uX]_d$. On aurait très bien pu prendre
$$
F^{\dagger\dagger} : \  
\bfA[\uX]_{d}^\star \ \longrightarrow \  \bfA[\uY]_{d'}
\qquad \qquad
(X^\alpha)^\star \ \longmapsto \ \sum_{|\beta| = d'} a_{\alpha \beta} Y^\beta
$$
D'ailleurs, cette application $F^{\dagger \dagger}$ n'est ni plus ni moins que 
la transposée de $F^\dagger$, \idest{} $F^{\dagger \dagger} = \transpose F^\dagger$.
Un autre visage de $F^{\dagger\dagger}$ est $G^\dagger$ où 
$G(\uX,\uY) = F(\uY,\uX) \in \bfA[\uX,\uY]_{d',d}$ ; en effet, on a 
$$
G^\dagger : \  
\bfA[\uY]_{d}^\star \ \longrightarrow \  \bfA[\uX]_{d'}
\qquad \qquad
(Y^\alpha)^\star \ \longmapsto \ \sum_{|\beta| = d'} a_{\alpha \beta} X^\beta
$$
ce qui la même application que $F^{\dagger\dagger}$ (modulo l'échange $\uX\leftrightarrow\uY$).

\begin{prop} \label{PermutationXYdansMho}
On reprend les notations de la remarque ``Pour les amateurs
d'isomorphismes monomiaux'' figurant après la définition
géante~\ref{DefGeante} en notant $\calL_{d,d'}
= \Jex_{1,d} \oplus \Smac_{0,d} \oplus \Jex^\star_{1,d'}$ le
$\bfA$-module monomial sur lequel opère tout endomorphisme $\mho_d(F)$
pour $F \in \bfA[\uX,\uY]_{d,d'}$.

Pour $F =F(\uX, \uY) \in \bfA[\uX, \uY]_{d,d'}$ et son cousin $G =
F(\uY,\uX) \in \bfA[\uX,\uY]_{d',d}$, on dispose du diagramme commutatif
suivant :
$$
\xymatrix @C=1.5cm @M=0.4pc {
\calL_{d',d} \ar[d]_{\rm iso} \ar[r]^{\mho_{d'}(G)} & \calL_{d',d} \ar[d]^{\rm iso} \\
(\calL_{d,d'})^\star \ar[r]_{{}^{\rm t} \mho_d(F)} & (\calL_{d,d'})^\star \\
}
$$
où $\mathrm{iso} : \calL_{d',d} \rightarrow (\calL_{d,d'})^\star = 
\Jex_{1,d}^\star \oplus \Smac_{0,d}^\star \oplus \Jex_{1,d'}$ 
est défini par 
$$
X^\alpha \oplus X^\gamma \oplus (Y^\beta)^\star \longmapsto
(Y^\beta)^\star  \oplus (Y^{\emouton-\gamma})^\star \oplus X^\alpha
$$
En particulier, \boxed{\varpi_{d'}(G) = \varpi_{d}(F)}.

\medskip

De manière matricielle, fixons des bases monomiales, $\calB_{1,d}$ de
$\Jex_{1,d}$ et $\calS_{0,d}$ de $\Smac_{0,d}$, idem pour $d'$.
Notons $\calB = \calB_{1,d} \vee \calS_{0,d} \vee \calB_{1,d'}^\star$
la base de $\calL_{d,d'}$ et
$\calB' = \calB_{1,d'} \vee \calS_{0,d'} \vee \calB_{1,d}^\star$ celle
de $\calL_{d',d}$ (une
étoile désigne la base duale).

\medskip

On dispose alors de la  conjugaison :
$$
\mathrm{Mat}_{\calB'}\big(\mho_{d'}(G)\big)
\ = \ P^{-1} \ \transpose{\mathrm{Mat}}_{\calB}\big(\mho_d(F)\big) \, P
$$
où $P$ est la matrice de $\mathrm{iso}$ dans les bases $\calB'$ et $\calB^\star$ ; 
c'est une matrice anti-diagonale en 3 blocs où les blocs extrêmes sont l'identité 
et le bloc du milieu est la matrice du mouton-swap exprimée dans les bases~$\calS_{0,d'}$ 
et $\calS_{0,d}^\star$ (c'est une matrice de permutation).
La matrice $P$ est elle-même une matrice de permutation.
\end{prop}


\begin{proof}

Le lecteur contemplera les schémas suivants. D'abord les endomorphismes:
$$
\def \taille {8}
\def \absM {2.4} 
\def \ordM {3.7} 
\def \milieu {(\absM, \ordM)}
\begin{tikzpicture}[scale = 0.5]
\draw[fill=gray!20]  (0,\taille) rectangle \milieu ;
\draw[fill = white]  (0,0) rectangle \milieu ;
\draw[fill=gray!20]  (\taille, 0) rectangle \milieu ;
\draw[fill=gray!60]  (\taille,\taille) rectangle \milieu ;
\draw[thick, dashed] (0,\taille) rectangle (\absM,\taille-\absM) node [midway] {$W_{1,d'}$} ;
\draw[thick, dashed] (\taille-\ordM, \ordM) rectangle (\absM,\taille-\absM) 
node [midway] {$G_{\hbox{\tiny\mouton}}^{\dagger}$} ; 
\draw[thick, dashed] (\taille-\ordM, \ordM) rectangle (\taille,0) node [midway] {$\transpose{W_{1,d}}$} ;
\path (\taille, \taille) -- (\taille,\taille-\absM) node[midway,right] {$\Jex_{1,d'}$} ;
\path (\taille, \ordM) -- (\taille,\taille-\absM) node[midway,right] {$\Smac_{0,d'}$} ;
\path (\taille, 0) -- (\taille,\ordM) node[midway,right] {$\Jex_{1,d}^\star$} ;
\path (0,0) -- (0, \taille) node[midway, left] {$\mho_{d'}(G) \ = \ $} ;
\end{tikzpicture}
\qquad
\def \absM {5.6} 
\def \ordM {4.3} 
\def \milieu {(\absM, \ordM)}
\begin{tikzpicture}[scale = 0.5]
\draw[fill=gray!20]  (0,\taille) rectangle \milieu ;
\draw[fill=gray!60]  (0,0) rectangle \milieu ;
\draw[fill=gray!20]  (\taille, 0) rectangle \milieu ;
\draw[fill = white]  (\taille,\taille) rectangle \milieu ;
\draw[thick, dashed] (0,\taille) rectangle (\taille-\ordM, \ordM) node [midway] {$\transpose{W_{1,d}}$} ;
\draw[thick, dashed] (\taille-\ordM, \ordM) rectangle (\absM,\taille-\absM) 
node [midway] {$\transpose F_{\hbox{\tiny\mouton}}^{\dagger}$} ; 
\draw[thick, dashed] (\taille, 0) rectangle (\absM,\taille-\absM) node [midway] {$W_{1,d'}$} ;
\path (\taille, \ordM) -- (\taille, \taille) node[midway,right] {$\Jex_{1,d}^\star$} ;
\path (\taille, \ordM) -- (\taille,\taille-\absM) node[midway,right] {$\Smac_{0,d}^\star$} ;
\path (\taille, 0) -- (\taille, \taille-\absM) node[midway,right] {$\Jex_{1,d'}$} ;
\path (0,0) -- (0, \taille) node[midway, left] {$\transpose{\mho_d(F)} \ = \ $} ;
\end{tikzpicture}
$$
Quant à l'isomorphisme
$$
{\rm iso} : \calL_{d',d} = \Jex_{1,d'} \oplus \Smac_{0,d'} \oplus \Jex_{1,d}^\star
\longmapsto
(\calL_{d,d'})^\star = \Jex_{1,d}^\star \oplus \Smac_{0,d}^\star \oplus \Jex_{1,d'}
$$
voici sa matrice dans les bases spécifiées:
$$
\def \taille {8}
\def \absM {2.4} 
\def \ordM {3.7} 
\begin{tikzpicture}[scale = 0.5]
\draw (0,0) rectangle (\taille,\taille) ;
\draw (0,0) rectangle (\absM,\absM) ;
\path (0,0) -- (\absM,\absM) node[midway] {$\mathrm{id}$} ;
\path (0,\taille) -- (\absM,\taille) node[midway,above] {$\Jex_{1,d'}$} ;
\path (\taille,0) -- (\taille,\absM) node[midway,right] {$\Jex_{1,d'}$} ;
\draw (\absM,\absM) rectangle (\taille-\ordM, \taille-\ordM) ;
\path (\absM,\absM) rectangle (\taille-\ordM, \taille-\ordM) node[midway] {$\mouton^{\rm sw}$} ;
\path (\absM,\taille) rectangle (\taille-\ordM, \taille) node[midway,above] {$\Smac_{0,d'}$} ;
\path (\taille,\absM) rectangle (\taille, \taille-\ordM) node[midway,right] {$\Smac_{0,d}^\star$} ;
\draw (\taille-\ordM,\taille-\ordM) rectangle (\taille, \taille) ;
\path (\taille-\ordM,\taille-\ordM) rectangle (\taille, \taille) node[midway] {$\mathrm{id}$} ;
\path (\taille-\ordM,\taille) rectangle (\taille, \taille) node[midway,above] {$\Jex_{1,d}^\star$} ;
\path (\taille,\taille-\ordM) rectangle (\taille, \taille) node[midway,right] {$\Jex_{1,d}^\star$} ;
\path (0,0) -- (0, \taille) node[midway, left] {$\mathrm{iso} \ = \ $} ;
\end{tikzpicture}
$$
\end{proof}

\begin{prop} \label{SymetrieVarpi}
Pour n'importe quel bezoutien $\Bez = \Bez(\uX,\uY)$ d'un système $\uP$,
on a l'égalité $\varpi_d(\Bez_{d,d'}) \, = \, \varpi_{d'}(\Bez_{d',d})$.
\end{prop}

\begin{proof}
Soit $\dsV(\uX,\uY)$ une matrice bezoutienne de déterminant $\Bez(\uX,\uY)$.
Alors $\dsV^*(\uX, \uY) = \dsV(\uY, \uX)$ est aussi une matrice bezoutienne de 
déterminant $\Bez^*(\uX,\uY) = \Bez(\uY, \uX)$.
On a $\Bez^*_{d,d'}(\uX,\uY) = \Bez_{d',d}(\uY, \uX)$, en insistant lourdement
sur le fait que les bidegrés sont en permanence relatifs à $\bfA[\uX,\uY]$.
D'après le théorème~\ref{IndependanceMatBez} d'indépendance du choix de la matrice bezoutienne,
on a donc
$$
\varpi_d\big(\Bez_{d,d'}(\uX, \uY)\big) 
\ = \ 
\varpi_d\big(\Bez_{d',d}(\uY,\uX)\big)
$$
En appliquant le trick $(\uX,\uY)$ versus $(\uY,\uX)$ de la
proposition~\ref{PermutationXYdansMho} à $F
= \Bez_{d',d}(\uY,\uX) \in \bfA[\uX,\uY]_{d,d'}$ et son cousin $G
= \Bez_{d',d}(\uX,\uY) \in \bfA[\uX,\uY]_{d',d}$, on en déduit
$$
\varpi_{d}\big(\Bez_{d',d}(\uY,\uX)\big) = \varpi_{d'}\big(\Bez_{d',d}(\uX,\uY)\big)
$$
En combinant ces deux égalités, on obtient l'égalité souhaitée.
\end{proof}

\subsection{L'application $\varpidres = \varpidRes{\protect\uP} : \bfA[\protect\uX,\protect\uY]_{d,d'} \to \bfA$}

\smallskip

Comme souvent, le système $\uP = (P_1, \dots, P_n)$ est implicite et
nous le mentionnerons très peu.  Faisons-le cependant une (seule)
fois, en utilisant la proposition d'\og \'Epuration\fg{}
(cf.~\ref{varpiDivisibilite}) qui permet, pour
$F \in \bfA[\uX,\uY]_{d,d'}$, de définir le quotient exact:
$$
\varpidRes{\uP}(F) \ := \ \dfrac{\varpi_{d,\uP}(F)}{\Delta_{2,d}(\uP)\,\Delta_{2,d'}(\uP)}
\ =\  \dfrac{\varpi_{d,\uP}(F)}{\det W_{2,d}(\uP) \times \det W_{2,d'}(\uP)}
$$
On perçoit ainsi la lourdeur de telles notations en comparaison de:
$$
\varpidres(F) \ =\ \dfrac{\varpi_d(F)}{\Delta_{2,d}\,\Delta_{2,d'}}
\ = \ \dfrac{\varpi_d(F)}{\det W_{2,d} \times \det W_{2,d'}}
$$
Précisons que nous allons nous placer en terrain générique évitant
ainsi tout souci dans les divisions (exactes).  Bien entendu,
$\varpidRes{\uP}(F)$ est défini pour n'importe quel système $\uP$ par
spécialisation du cas générique. Par ailleurs, comme nous devons
aborder l'étude du poids en $P_i$ de $\varpidres$, le terrain
générique nous sera utile même s'il n'est pas indispensable.

\label{NOTA17-varpidres}%

\smallskip

D'après~\ref{varpiStructure}, ce scalaire $\varpidres(F)$ est une somme signée de termes du type :
$$
\dfrac{\text{mineur d'ordre $s_d$ de ${\Syl_d}_{|\Smac_{1,d}}$}}{\det W_{2,d}} 
\times 
\begin {array}{c}
\text{mineur d'ordre}\\
\text{$\chi_d$ de $F^{\dagger}$}\\
\end {array}
\times 
\dfrac{\text{mineur d'ordre $s_{d'}$ de $\transpose{\big({\Syl_{d'}}_{|\Smac_{1,d'}}}\big)$}}
{\det\transpose{W_{2,d'}}}
\leqno (\diamond)
$$

\begin{prop} \label{varpiresPropriete}
\leavevmode
\begin{enumerate}[\rm i)]
\item
On a $\varpidres\big((XY)^\emouton_{d,d'}\big) = \Delta_{1,d} \times \Delta_{1,d'}$.
Ce scalaire est homogène en $P_i$ de poids $\dim\Jex_{1\setminus 2, d}^{(i)} +
\dim\Jex_{1\setminus 2, d'}^{(i)}$.

\item
Soit $\Bez = \Bez(\uX,\uY)$ un bezoutien de $\uP$.
Alors $\varpidres(\Bez_{d,d'})$ est homogène en $P_i$ de poids $\widehat d_i$.

\item
Pour le jeu étalon $\uX^D$, on a $\varpidRes{\uX^D}(\Bez_{d,d'}) = 1$.
\end{enumerate}
\end{prop}

\begin{proof}\leavevmode

\medskip  
i) La proposition \ref{varpiEvalEtalon} fournit 
$\varpi_d\big((XY)^\emouton_{d,d'}\big) = \det W_{1,d} \times \det W_{1,d'}$.
D'où le résultat en utilisant $\Delta_{1,d} = \det W_{1,d} / \det W_{2,d}$ pour
tout $d$ et la définition de $\varpidres$.

\medskip
ii) Grâce à l'égalité $(\diamond)$ et au point i), il suffit de
déterminer le poids de~$\Bez_{d,d'}^\dagger$.  Par construction, une
matrice bezoutienne possède une unique colonne relative à $P_i$ et
pour laquelle chaque coefficient est homogène en $P_i$ de poids~1.
Par conséquent, chaque coefficient de la matrice~$\Bez_{d,d'}^\dagger$
(qui est le coefficient d'un certain $X^\alpha Y^\beta$ de $\Bez$) est
homogène en $P_i$ de poids $1$.  Bref, $\varpidres(\Bez_{d,d'})$ est
homogène en $P_i$ de poids $\dim \Jex_{1\setminus 2,d}^{(i)} + \chi_d
+ \dim \Jex_{1\setminus 2,d'}^{(i)}$ et cette somme vaut
exactement~$\widehat{d_i}$ en vertu du lemme qui suit.

\medskip
iii)
Utiliser~\ref{varpiEtalon} ou bien le point i) sachant que $\Delta_{1,d}(\uX^D) = 1$
pour tout $d$.
\end{proof}

\begin{lem}\label{varpiresPoids}

Notons $\bfA[\widehat{X_i}] = \bfA[X_1,\dots,X_{i-1},\,
X_{i+1}, \dots, X_n] \subset \bfA[\uX]$. Alors
le $\bfA$-module $\Jex_{1\setminus 2,d}^{(i)} \oplus \Smac_{0,d} \oplus \Jex_{1\setminus 2,d'}^{(i)}$ 
est monomialement isomorphe au $\bfA$-module $\Smac_0 \cap \bfA[\widehat{X_i}]$.
En conséquence:
$$
\dim \Jex_{1\setminus 2,d}^{(i)} 
\ + \
\chi_{d} 
\ + \ 
\dim \Jex_{1\setminus 2,d'}^{(i)} 
\ = \ 
\widehat{d_i}
$$
\end{lem}

\begin{proof}
Posons $T_d = \Jex_{1\setminus 2,d}^{(i)} \oplus \Smac_{0,d}$ et considérons
les deux applications monomiales suivantes:
$$
\begin{array}[t]{rcl}
T_d \oplus \Jex_{1\setminus 2,d'}^{(i)}
& \longrightarrow & \Smac_0\cap\bfA[\widehat{X_i}] \\ [0.5cm]
X^\alpha \oplus 0
& \longmapsto & \dfrac{X^\alpha}{X_i^{\alpha_i}} \\ [0.5cm]
0 \oplus X^{\alpha'} & \longmapsto & \dfrac{X^{\emouton-\alpha'}}{X_i^{d_i-1-\alpha'_i}} \\
\end{array}
\hspace{1cm}
\begin{array}[t]{rcl}
\Smac_0\cap\bfA[\widehat{X_i}] &\longrightarrow &T_d \oplus \Jex_{1\setminus 2,d'}^{(i)}
\\[0.5cm]
X^\gamma & \longmapsto & 
\begin{cases}
X_i^{d - |\gamma|} X^\gamma \ \oplus \ 0 
& \text{si } d-|\gamma| \geqslant 0 \\
\\
0 \ \oplus \ X_i^{|\gamma|-d} \, X^{\emouton-\gamma}
& \text{si } d - |\gamma| < 0 \\
\end{cases}
\end{array}
$$
Nous affirmons qu'elles sont bien définies. Par exemple, à gauche, en bas, on a $|\alpha'| \leqslant \delta$ donc
$$
\alpha'_i \leqslant d_i-1 + \sum_{j\ne i} \big((d_j-1)-\alpha'_j\big)  \leqslant d_i-1
$$
l'inégalité de droite étant due au fait que $\alpha'_j \leqslant d_j-1$ pour $j \ne i$ 
puisque $X^{\alpha'} \in \Jex_{1\setminus2,d'}^{(i)}$.
Nous laissons le soin au lecteur de vérifier les autres détails et le fait 
que les deux applications ci-dessus sont bien réciproques l'une de l'autre. 

\medskip
En ce qui concerne $\dim\big(\Smac_0\cap \bfA[\widehat{X_i}]\big) = \widehat{d}_i$, cf le premier
point de~\ref{PoidsJi12dS0d}.
\end{proof}

\begin{proof}[Autre preuve de l'égalité dimensionnelle sans utiliser $\bfA\lbrack\widehat{X_i}\rbrack$]
\leavevmode

\smallskip
C'est l'occasion de faire joujou avec les séries, les fractions rationnelles et les polynômes
en une indéterminée $t$, à coefficients entiers.

\smallskip
$\rhd$
Définissons le \emph{polynôme} $\chi \in \bbZ[t]$, unitaire de degré $\delta$:
$$
\chi = \chi(t)
\ = \ 
\prod_{j=1}^n \dfrac{1-t^{d_j}}{1-t}
\ = \ 
\prod_{j =1}^n (1+t +t^2 +\cdots + t^{d_j-1})
$$
Par définition, le coefficient en $t^d$ de $\chi$ est la dimension de $\Smac_{0,d}$, 
ou encore la caractéristique d'Euler-Poincaré $\chi_{d}$ du complexe $\rmK_{\sbullet,d}(\uX^D)$.

\smallskip
$\rhd$
Quant à la série $S^{(i)}$ de $\Jex_{1\setminus 2}^{(i)}$, elle est égale à la fraction rationnelle
(cf.~\ref{SerieJ1minusJ2i}):
$$
S^{(i)} = S^{(i)}(t)
\ = \ 
\dfrac{t^{d_i}}{1-t} \ \prod_{j \neq i} \dfrac{1-t^{d_j}}{1-t}
\ = \
\dfrac{U_i(t)}{1-t}
\qquad \text{où} \qquad
U_i(t) = t^{d_i} \,\prod_{j \neq i} (1+t +t^2 +\cdots + t^{d_j-1})
$$
En ayant remarqué que $\deg U_i=\delta+1$ et $U_i(1) = \widehat d_i$, écrivons:
$$
S^{(i)} = S^{(i)}_{\leqslant\delta} \, +\,  S^{(i)}_{\geqslant\delta+1}
\qquad \text{avec} \qquad
S^{(i)}_{\leqslant \delta} = \dfrac{U_i(t) - \widehat d_i\, t^{\delta+1}} {1-t},
\qquad
S^{(i)}_{\geqslant\delta+1} =  \widehat{d_i}\,\dfrac{t^{\delta+1}}{1-t}
$$
Cette écriture montre que $S^{(i)}_{\leqslant \delta}$ est un
\textit{polynôme} de degré $\leqslant \delta$ (le coefficient en
$t^\delta$ de ce polynôme est~$\widehat {d_i} - 1$) et que les
coefficients d'ordre $\geqslant \delta+1$ de la série $S^{(i)}$ sont
égaux à $\widehat d_i$.
Pour $0 \leqslant d \leqslant \delta$,
l'entier $\dim \Jex_{1\setminus 2,d}^{(i)}$ est le coefficient en
$t^d$ de $S^{(i)}_{\leqslant\delta}$.

\smallskip
$\rhd$
Il nous faut maintenant un \emph{truc} pour accéder au coefficient en
$t^{\delta-d}$ de $S^{(i)}_{\leqslant\delta}$, un truc du genre
polynôme réciproque.  A cet effet, on introduit une involution sur les
fractions rationnelles $F = F(t)$ en posant $F^\star(t) =
t^\delta\,F(1/t)$.  Et effectivement, si $F$ est un polynôme de degré $\leqslant \delta$, il
en est de même de $F^\star$, polynôme réciproque de $F$. Ainsi,
$\dim \Jex_{1\setminus 2,\delta-d}^{(i)}$ est le coefficient
en $t^d$ du polynôme $\big(S^{(i)}_{\leqslant\delta}\big)^\star$.

\smallskip
$\rhd$
Pour montrer l'égalité dimensionnelle, il suffit d'établir l'égalité polynomiale:
$$
S^{(i)}_{\leqslant \delta} \ +\  \chi \ +\  
\big(S^{(i)}_{\leqslant \delta}\big)^\star \ = \ 
\sum_{d=0}^\delta \ \widehat{d_i} \, t^d
$$
Or on peut exprimer les trois \textit{polynômes} de gauche en fonction 
de la \textit{fraction rationnelle} $S^{(i)}$ :
$$
S^{(i)}_{\leqslant \delta} \ =\  
S^{(i)} \,-\, \widehat{d_i}\,\dfrac{t^{\delta+1}}{1-t}, 
\qquad 
\chi \ = \ 
\dfrac{1 -t^{d_i}}{t^{d_i}}\,S^{(i)}
\qquad \text{et} \qquad 
\big(S^{(i)}_{\leqslant \delta}\big)^\star 
\ = \ 
-\dfrac{S^{(i)}}{t^{d_i}}  \,+\, \widehat{d_i} \dfrac{1}{1-t}
$$
Dans leur somme, les contributions des $S^{(i)}$ s'éliminent, 
et il reste $\widehat{d_i} \dfrac{1-t^{\delta+1}}{1-t}$ comme voulu.

\smallskip
$\rhd$
Une remarque concernant la formule donnée pour $\big(S^{(i)}_{\leqslant \delta}\big)^\star$.
Elle peut être obtenue à partir de l'égalité donnée pour $S^{(i)}_{\leqslant \delta}$,
des formules:
$$
\big(S^{(i)})^\star = -\dfrac{S^{(i)}}{t^{d_i}}, \qquad
\left(\dfrac{t^{\delta+1}}{1-t}\right)^\star = \dfrac{-1}{1-t}
\leqno(\diamond)
$$
et du fait que l'involution $F \mapsto F^\star$ est additive. Dans $(\diamond)$, l'égalité
à droite est immédiate; pour celle de gauche, utilisons la propriété de \og multiplicativité\fg{}
de l'involution $F \mapsto F^\star$:
$$
(FG)^\star = F(1/t)\,G^\star
$$
et le fait qu'elle laisse invariant $\chi$ (ce qui traduit $\chi_d = \chi_{\delta-d}$ pour $0\leqslant d\leqslant\delta$).
En écrivant $\chi = (t^{-d_i}- 1)\,S^{(i)}$
$$
\chi^\star = (t^{d_i} - 1)\,(S^{(i)})^\star
\qquad \text{d'où} \qquad
(S^{(i)})^\star = \dfrac{\chi^\star}{t^{d_i}-1} = \dfrac{\chi}{t^{d_i}-1} =
-\dfrac{S^{(i)}}{t^{d_i}}
$$
Bon, assez joué, la récréation est terminée.
\end{proof}

\begin{rmq}

Il est amusant de noter que $\varpidres(F)$ est le quotient du
déterminant de $\mho_d(F)$ par un certain mineur de $\mho_d(F)$.  Pour
le justifier, notons $W_{2,(d,d')}(F)$ l'induit-projeté de $\mho_d(F)$
sur le sous-module monomial $\Jex_{2,d} \oplus \Jex_{2,d'}^\star$ de
$\bfA[\uX]_d \oplus \Jex_{1,d'}^\star$. Le mineur en question
est le déterminant de cet induit-projeté.

Ce qui est un peu \og dommage\fg{}, c'est que $W_{2,(d,d')}(F)$ dépend
de $\uP$ mais également de $F$; mais son déterminant $\det
W_{2,(d,d')}(F)$, n'en dépend pas (de $F$) puisqu'égal à $\det
W_{2,d} \times \det W_{2,d'}$. On peut donc se permettre d'écrire:
$$
\varpidres(F) = \dfrac{\det\mho_d(F)}{\det W_{2,(d,d')}}
$$
L'aspect amusant est une allusion à la formule classique de Macaulay dont
nous allons voir la généralisation:
$$
\Res(\uP) = \dfrac{\det\mho_d(\Bez_{d,d'})}{\det W_{2,(d,d')}}
$$
Mais peut-être que ``amusant'' est un peu excessif? Comme nous ne sommes pas persuadés
que cela amusera le lecteur, nous lui glissons un petit schéma.

\medskip

\renewcommand \absM {\vabsM}
\renewcommand \ordM {\vordM}
\newcommand \absN {2.3}
\newcommand \ordN {1.3}
\newcommand \Wdeux {(\absN, \taille-\absN)} 
\newcommand \tWdeux {(\taille-\ordN ,\ordN)}

\begin{tikzpicture}[scale = 0.6]
\draw[fill=gray!20]  (0,\taille) rectangle \milieu ;
\draw[fill=gray!60]  (\taille,\taille) rectangle \milieu ;
\draw[fill=gray!20]  (\taille, 0) rectangle \milieu ;
\draw  (0,0) rectangle \milieu ;
\draw[thick, dashed] (0,\taille) rectangle \Wdeux ;
\path (0,\taille) -- \Wdeux node[midway] {\hbox{\tiny $W_{2,d}$}} ;
\draw[thick, dashed] \tWdeux rectangle (\taille, 0) ;
\path \tWdeux -- (\taille, 0) node[midway] {\hbox{\tiny $\transpose W_{2,d'}$}} ;
\path (0,0) -- (0, \taille) node[midway, left] {$\mho_d(F) \ = \ $} ;
\end{tikzpicture}
\hspace{2cm}
\begin{tikzpicture}[scale = 0.7]
\draw[fill=gray!60] (\absN+\ordN, \absN+\ordN) rectangle (\absN, \ordN) ;
\draw  (0,0) rectangle (\absN, \ordN) ;
\draw[thick, dashed, fill=gray!20] (0,\absN+\ordN) rectangle (\absN, \ordN)  ;
\path (0,\absN+\ordN) -- (\absN, \ordN) node[midway] {\hbox{\tiny $W_{2,d}$}} ;
\draw[thick, dashed,fill=gray!20] (\absN, \ordN)  rectangle (\absN+\ordN, 0) ;
\path (\absN, \ordN) -- (\absN+\ordN, 0) node[midway] {\hbox{\tiny $\transpose W_{2,d'}$}} ;
\path (\absN+\ordN, \absN+\ordN) -- (\absN+\ordN, \ordN) node[midway,right] 
{\  \footnotesize $ \Jex_{2,d}$} ;
\path (\absN+\ordN, \ordN) -- (\absN+\ordN, 0) node[midway,right] 
{\  \footnotesize $\Jex_{2,d'}^\star$} ;
\path (0,0) -- (0, \absN+\ordN) node[midway, left] {$W_{2,(d,d')}(F) \ = \ $} ;
\end{tikzpicture}

\end{rmq}

\subsection{La formule Graal pour le résultant}

Nous continuons à désigner par $\Bez = \Bez(\uX,\uY)$ un bezoutien quelconque
du système $\uP = (P_1, \cdots, P_n)$.

\begin{theo}\label{FormuleGraal}

Pour tout $d+d' = \delta$, on dispose de l'égalité que nous qualifions de \og formule Graal \fg{}:
$$
\Res(\uP) = \varpidres(\Bez_{d,d'})
$$
C'est-à-dire :
$$
\Res(\uP) \ = \  \dfrac{\varpi_d(\Bez_{d,d'})}{\Delta_{2,(d,d')}}
\overset {\rm def.}{=}
\dfrac{\det \mho_d\big(\Bez_{d,d'}\big)}{\det W_{2,d} \times \det W_{2,d'}}
$$
Plus généralement:
$$
\Res(\uP) = \dfrac{\det\mho_d^\sigma\big(\Bez_{d,d'}\big)}
                  {\det W^\sigma_{2,d} \det W^\sigma_{2,d'}} \qquad
\forall\ \sigma\in\fS_n
$$
\end{theo}

\begin{proof}
Il suffit de faire la preuve en générique. 
D'après~\ref{varpiElimIdeal}, 
$$
\det W_{2,d}\, \det W_{2,d'} \, \varpidres(\Bez_{d,d'}) 
\ \in \ \ElimIdeal = \langle \Res(\uP) \rangle
$$
Mais $\det W_{2,d}$, pour n'importe quel $d$, fait partie des scalaires réguliers
modulo~$\uPsat$ (cf le théorème~\ref{GenPsatRegScalars}).
Ainsi il existe $\lambda \in \bfA = \bfk[\indetsPi]$ tel que 
$\varpidres(\Bez_{d,d'})= \lambda \, \Res(\uP)$.
Et dans la section prédédente, très exactement
en~\ref{varpiresPropriete}, nous avons préparé le terrain en montrant
que $\varpidres(\Bez_{d,d'})$ est homogène en $P_i$ de poids $\widehat
d_i$.  Il en est de même de $\Res(\uP)$: ce n'est pas une surprise car
nous l'avons utilisé un certain nombre de fois depuis le
chapitre~\ref{ChapMacRaeForP} (en particulier depuis le théorème de
fondement~\ref{PoidsNormalisationMacRae}).  Enfin, comme il fallait
s'en douter, ces deux scalaires sont normalisés en le jeu étalon, donc
$\lambda = 1$.

\end{proof}

\begin{rmq}
La formule donnée ci-dessus fonctionne même si $d$, $d'$ ne sont pas
coincés entre~$0$ et $\delta$, étant entendu que l'on continue à imposer
$d+d' = \delta$.

\medskip
$\rhd$
Par exemple pour $d'=-1$ (donc $d = \delta+1$), on retrouve la formule à la Macaulay 
en degré~$\delta+1$: 
$$
\varpi^\res_{\delta+1}(\Bez_{\delta+1,\, -1}) \ = \ 
\dfrac{\det W_{1,\delta+1}}{\det W_{2,\delta+1}}
$$
Cf~\ref{MegaMacaulayTheorem} en remarquant qu'ici l'argument
$F$ de $\mho_{\delta+1}$ vaut 0 parce que $F \in
\bfA[\uX,\uY]_{\delta+1,-1} = \{0\}$ et que l'on l'égalité
d'endomorphismes:
$$
\mho_{\delta+1}(F) = W_{1,\delta+1}
$$

\medskip
$\rhd$
Pour $d' = 0$ (donc $d = \delta$), comme $\Bez_{\delta,0}$ \fbox{\emph{est}}
un déterminant $\uX$-bezoutien de $\uP$ (cf la première remarque
en \ref{BezoutianRemarks}), on peut l'affubler du nom $\nabla$
(nom de baptême depuis le chapitre~\ref{ObjetsSuiteP}), si bien que l'on retrouve la formule en degré $\delta$
$$
\varpi^\res_\delta(\Bez_{\delta,0}) \ = \ 
\dfrac{\omega(\nabla)}{\det W_{2,\delta}}
$$
On peut invoquer cf~\ref{MegaMacaulayTheorem} sans oublier 
que l'on l'égalité d'endomorphismes de $\bfA[\uX]_\delta$:
$$
\mho_\delta(F) = \Omega(F) \qquad  \forall F \in \bfA[\uX]_\delta
$$

\medskip
$\rhd$
Pour $d' < \min (D)$, en utilisant le lien entre bezoutien et
déterminant de Sylvester
i.e. $\Bez_\beta \equiv \bsnabla_\beta \bmod \langle\uP\rangle_d$ pour
$|\beta| = d'$ (notations et contenu de~\ref{LienSylvesterBezoutien}),
on retrouve la formule donnée en~\ref{ResViaDetd}
$$
\varpidres(\Bez_{d,\, d'}) \ = \ 
\dfrac{\det \swbsOmega_d\big((\bsnabla_\beta)_\beta\big)}{\det W_{2,d}}
$$
Mentionnons également l'égalité d'endomorphismes de $\bfA[\uX]_d$ pour
$F=\sum_{|\beta|=d'} F_\beta(\uX)Y^{\beta} \in \bfA[\uX,\uY]_{d,d'}$:
$$
\mho_d(F) = \bsOmega_d(f)  \qquad
f \in \AXdSod \qquad
f(X^\alpha) = F_{\emouton-\alpha}(\uX) \in \bfA[\uX]_d
$$
Cette section peut donc être vue comme une géante généralisation de tout 
ce qui précède.
\end{rmq}

\begin{prop}
La dimension de $\calL_{d,d'}$, module sur lequel opère $\mho_d(F)$ pour $F \in \bfA[\uX,\uY]_{d,d'}$, 
est minimale pour $d = \lfloor \frac{\delta}{2} \rfloor$.
Si $\delta$ est impair, ce minimum est également atteint en 
$d = \lfloor \frac{\delta}{2} \rfloor + 1$.
\end{prop}

\begin{proof}

L'espace en question est $\bfA[\uX]_d \oplus \Jex_{1,d'}^\star
= \Jex_{1,d} \oplus \Smac_{0,d} \oplus \Jex_{1,d'}^\star$, de dimension
$$
s_d + \chi_{d} + s_{d'} 
\ = \ 
(\dim \bfA[\uX]_d - \chi_{d})
\ + \ 
\chi_{d}
\ +\ 
(\dim \bfA[\uX]_{d'} -  \chi_{d'}\big)
$$
c'est-à-dire, en notant $b_d = \dim \bfA[\uX]_d$ et en utilisant $\chi_d = \chi_{d'}$, de dimension 
$$
b_d + b_{d'} - \chi_{d}
$$
On a vu en \ref{chiProprieteComplements} que $d \mapsto - \chi_{d}$
est minimale pour $d = \lfloor\frac{\delta}{2}\rfloor$.  Reste à
étudier les variations de $d \mapsto b_d + b_{\delta -d}$.

On rappelle que 
$$
b_d = \binom{n-1+d}{n-1} 
= \dfrac{1}{(n-1)!} \prod_{j=1}^{n-1} (d+j)
$$
ce qui conduit à considérer le polynôme $b = b(X) \in \bbQ[X]$:
$$
b = \dfrac{1}{(n-1)!} \prod_{j=1}^{n-1} (X+j)
$$
C'est un polynôme à coefficients positifs. Il en est donc de même de toutes ses dérivées.
En particulier, la fonction polynomiale $b'$ est croissante sur $\bbR^+$.

Introduisons le polynôme $g$ défini par $g(X) = b(X) + b(\delta-X)$,
de dérivée $g'(X) = b'(X) - b'(\delta-X)$. Pour $x$ dans l'intervalle
$[\delta/2\, , \, {+\infty}[$, ce qui équivaut à $x \ge \delta-x$, on a:
$$
g'(x) = b'(x) - b'(\delta-x) \ge 0   \quad \text{car $b'$ est croissante sur $\bbR^+$}
$$
Ainsi la fonction polynomiale $g$ est croissante sur $[\delta/2\, , \, {+\infty}[$.
Comme elle est invariante par $x \mapsto \delta-x$, elle est décroissante sur
$]{-\infty}\, , \,\delta/2]$. Bilan: elle admet un minimum en $x = \delta/2$.

\medskip
\noindent
Conclusion sur les variations de $d \mapsto b_d + b_{\delta - d}$.

Si $\delta$ est pair, alors le minimum de $d \mapsto b_d + b_{\delta - d}$ est atteint en 
$\frac{\delta}2$ qui est aussi $\lfloor\frac{\delta}{2}\rfloor$.

Si $\delta$ est impair, les deux entiers 
$\lfloor\frac{\delta}{2}\rfloor$ et 
$\lfloor\frac{\delta}{2}\rfloor + 1$ sont symétriques par rapport à 
$x = \frac{\delta}{2}$ et alors le minimum de $d \mapsto b_d + b_{\delta - d}$ est atteint en ces deux entiers.
\end{proof}

\medskip

La proposition qui suit fait écho à~\cite[3.11.19.22]{J7}.
Rappelons (cf.~\ref{sectionProfondeurWh}) que $\Sigma_2$
désigne n'importe quelle partie de $\mathfrak S_n$ telle que 
$$
\big\{ \sigma(n), \sigma \in \Sigma_2\big\} = \{1..n\}
$$
Par exemple, pour $\Sigma_2$ on peut prendre la famille des transpositions $(i,n)_{1 \leqslant i 
\leqslant n}$.

\begin{prop}
Pour $\uP$ couvrant le jeu étalon généralisé, on a :
$$
\Gr\big(\Delta_{2,(d,d')}^\sigma \,; \sigma \in \Sigma_2 \big) \ \geqslant \ 2
$$
Par conséquent, le résultant est le pgcd fort des $\varpi^\sigma_d(\Bez_{d,d'})$ pour 
$\sigma \in \Sigma_2$.
\end{prop}

\begin{proof} \leavevmode

Par définition, $\Delta^\sigma_{2,(d,d')} = \det W^\sigma_{2,d}\,\det
W^\sigma_{2,d'}$. On ne peut pas utiliser directement le
théorème~\ref{ProfondeurWh} qui stipule que
$$
\Gr\big (\det W_{2,d}^\sigma \,; \sigma \in \Sigma_2\big) \ge 2
$$
On peut cependant reprendre les ingrédients de la preuve (ou celle analogue de~\ref{ProfondeurBk}).
Cela consiste à considérer dans un premier temps le jeu étalon généralisé, puis à utiliser le 
contrôle de la profondeur par les composantes homogènes dominantes.

\medskip

\index{idéal!de Veronese}%

Du fait que $\uP$ couvre le jeu étalon généralisé $\pXD = (p_1X_1^{d_1}, \dots, p_nX_n^{d_n})$,
l'anneau $\bfA$ des coefficients de $\uP$ est un anneau de polynômes
de la forme $\bfA = \bfR[p_1, \cdots, p_n]$. Nous allons faire intervenir
l'idéal de Veronese $\scrV_{n-1}$ engendré par les monômes en $p_1,\dots,p_n$,
sans facteur carré, de degré $n-1$:
$$
\scrV_{n-1} = \langle p_J  \mid \#J = n-1\rangle \qquad\quad
p_J = \prod_{j \in J} p_j
$$
Notons $I = \{1 .. n-1\}$. D'après la propriété vérifiée par $\Sigma_2$:
$$
\scrV_{n-1} = \langle p_{\sigma(I)}  \mid \sigma \in \Sigma_2\rangle
$$
$\rhd$ Pour le jeu étalon généralisé $\pXD$:
$$
\Delta_{2,(d,d')}^\sigma (\pXD)
\ \overset{\rm def}{=} \ 
\det W_{2,d}^\sigma(\pXD) \times \det W_{2,d'}^\sigma (\pXD)
\ = \ 
\prod_{i \in \sigma(I)} p_i^{\nu_{i,d}+\nu_{i,d'}}  
$$
où $\nu_{i,d}$ est la dimension du $\bfA$-module 
de base les $X^\alpha \in \Jex_{2,d}$ tels que $\sminDiv(X^\alpha) = i$.
En prenant $e$ plus grand que tous les $\nu_{i,d} + \nu_{i,d'}$,
$i$ variant dans $\sigma(I)$ et $\sigma \in \Sigma_2$, on a l'inclusion:
$$
\fa := \big\langle p_{\sigma(I)}^e \mid \sigma \in \Sigma_2 \big\rangle
\quad\subset\quad
\fb := \big\langle \Delta_{2,(d,d')}^\sigma (\pXD) \mid \sigma \in \Sigma_2 \big\rangle
$$
D'après le théorème~\ref{ProfondeurVeronese}, on a l'inégalité $\Gr( \scrV_{n-1} ) \geqslant 2$.
En utilisant le corollaire~\ref{GrCompatibleInclusionRacine}, on en déduit que
$\Gr(\fa) \ge 2$, a fortiori $\Gr(\fb) \ge 2$.

\medskip
$\rhd$
Pour $\uP$ couvrant $\pXD$, la composante homogène dominante de $\Delta_{2,(d,d')}^\sigma(\uP)$
est $\Delta_{2,(d,d')}^\sigma(\pXD)$, ce qui permet de conclure en utilisant le contrôle de la profondeur par
les composantes homogènes dominantes, cf. le théorème~\ref{ControleProfondeur}.
\end{proof}

\begin{prop}
La formule Graal $\varpidres(\Bez_{d,d'})$ 
est une formule déterminantale (\idest{} \og sans dénominateur \fg{}) pour le résultant
si et seulement si
$$
\delta - m < d < m  \qquad \text{où} \quad m = \min_{i\neq j} (d_i + d_j)
$$
Supposons $d_1 \leqslant d_2 \leqslant \cdots \leqslant d_n$.
L'intervalle des $d$ ci-dessus est non vide  si et seulement si 
$$
d_3 + \cdots + d_n \ \leqslant \ d_1 + d_2 + n - 2
\leqno (\star)
$$
Pour $n=2$, tout format $D$ convient.

\smallskip 
\noindent
Pour $n=3$, le format de degrés est du type $D = (d_1, d_2, d_3)$ avec $d_3 \leqslant d_1+d_2+1$.

\smallskip
\noindent
Pour $n = 4$, il y a quatre cas :
$$
D = (d_1, d_1, d_1+1, d_1+1) 
\quad 
D = (d_1, d_1, d_1, d_1+2) 
\quad
\left \{
\begin{array}[c]{l}
D = (d_1, d_2, d_2, d_2+1) \\
d_2 \leqslant d_1 + 1
\end{array}
\right .
\quad
\left \{
\begin{array}[c]{l}
D = (d_1, d_2, d_2, d_2) \\
d_2 \leqslant d_1 + 2
\end{array}
\right .   
$$
Pour chaque $n \geqslant 5$, il y a quatre formats :
$$
D = (1,\dots,1,d_n) \text{\quad avec $d_n \leqslant 3$,} \qquad \qquad D= (1,\dots,1,2,2) 
$$
auxquels il faut ajouter pour $n=5$:
$$
D=(1,2,2,2,2), \quad D=(2,2,2,2,2), \quad D = (2, 2, 2, 2, 3), \quad D = (3,3,3,3,3)
$$
et pour $n=6$ le format $D = (2,2,2,2,2,2)$.
\end{prop}

\begin{proof}
Au dénominateur de la formule $\varpidres(\Bez_{d,d'})$ intervient
$\det W_{2,d} \times \det W_{2,d'}$.  En générique, ce terme vaut $1$
si et seulement si les matrices sont vides
\idest{} $\Jex_{2,d} = 0$ et $\Jex_{2,d'} = 0$.
Or d'après~\ref{NulliteJhd} cela équivaut à $d < m$ et $d' < m$, 
c'est-à-dire à $\delta - m < d < m$.
Si $d_1 \leqslant d_2 \leqslant \cdots \leqslant d_n$, la condition s'écrit 
$d_3 + \cdots + d_n -n < d < d_1 + d_2$. 
Par conséquent, l'intervalle des $d$ est non vide si et seulement si
l'inégalité~$(\star)$ est vérifiée.

\medskip

\noindent
$\rhd$ Pour $n=2$ et $n=3$, on obtient directement ce qui est indiqué.

\noindent
$\rhd$ Pour $n=4$, l'inégalité $(\star)$ est  $d_3 + d_4 \leqslant d_1 + d_2 + 2$.
On en déduit $d_2 + d_2 \leqslant d_3 + d_4 \leqslant d_2 + d_2 + 2$ donc
la somme $d_3 + d_4$ vaut soit $2d_2$, soit $2d_2 +1$ soit $2d_2+2$.
En étudiant ces cas, on tombe sur les formats annoncés.

\medskip
\noindent
Quelques inégalités générales avant de continuer. Reprenons 
l'inégalité $(\star)$ : $d_3 + \cdots + d_n \leqslant d_1 + d_2 + n-2$.
Les entiers de gauche peuvent être minorés par $d_3$ 
et à droite $d_1 + d_2 \leqslant 2d_3$, donc
$(n-2) d_3 \leqslant 2d_3 + n-2$ c'est-à-dire 
$(n-4) d_3 \leqslant n-2$.
Pour $n \geqslant 5$, on a $d_3 \leqslant \frac{n-2}{n-4}$. 
Ainsi,
\begin{center}
pour $n=5$, $d_3 \leqslant 3$ ;
\qquad 
pour $n=6$, $d_3 \leqslant 2$ ;
\qquad 
pour $n\geqslant 7$, $d_3 \leqslant 1$ ;
\end{center} 

\noindent
Supposons $d_3=1$. Donc $d_1=d_2=1$ et l'inégalité $(\star)$ est $d_4
+ \cdots + d_n \leqslant n-1$.  La somme de gauche étant $\geqslant
n-3$, tous les $d_i$ valent $1$, sauf éventuellement deux d'entre eux
qui valent $2$ ou un d'entre eux qui vaut $3$.

\noindent
Supposons $d_3 = 2$. Donc d'une part $d_1,d_2 \le 2$ et d'autre part
$d_i \ge 2$ pour $i \geqslant 4$. L'inégalité $(\star)$ a comme conséquence
$2(n-3) \leqslant d_4 + \cdots + d_n \leqslant d_1 + d_2 + n-4 \leqslant n$.

\medskip
\noindent
$\rhd$ Pour $n \geqslant 7$, on a $d_3 = 1$ et le cas est déjà traité.

\medskip
\noindent
Reste à étudier les cas $n=6$ et $n=5$ lorsque $d_3 \neq 1$.

\medskip
\noindent
$\rhd$ Pour $n = 5$, on a $d_3 \leqslant 3$. 

\noindent
Pour $d_3 = 3$, on obtient $d_3 + d_3 \leqslant d_4 + d_5 \leqslant d_1 + d_2 \leqslant d_3 + d_3$, 
d'où $d_1 = d_2 = d_4 = d_5 = d_3 = 3$.

\noindent
Pour $d_3 = 2$, on obtient $4 \leqslant d_4 + d_5 \leqslant d_1 + d_2 + 1 \leqslant 5$. 
Ainsi ou bien $d_4 + d_5 = 5$ auquel cas $D = (2, 2, 2, 2, 3)$ ou bien $d_4 + d_5 = 4$ auquel cas 
$D=(1,2,2,2,2)$ ou $D=(2,2,2,2,2)$.

\medskip

\noindent
$\rhd$ Pour $n=6$, on a $d_3 \leqslant 2$ donc $d_3 = 2$.

\noindent
On obtient $6 \leqslant d_4 + d_5 + d_6 \leqslant d_1 + d_2 + 2 \leqslant 6$, 
donc $d_1 = d_2 = d_4 = d_5 = d_6 = d_3 = 2$.
\end{proof}

\begin{rmq}
Pour un format de degrés du type $D = (e, \dots, e)$ avec $e \geqslant 1$, 
il existe une formule Graal déterminantale dans les cas suivants 
$$
n=2, \qquad 
n=3, \qquad 
n=4, \qquad 
n=5 \text{ et $e\leqslant 3$}, \qquad 
n=6 \text{ et $e \leqslant 2$}, \qquad 
n\geqslant 7 \text{ et $e = 1$}
$$
\end{rmq}

\subsection{Le bezoutien affine et une formule \og à la Bezout\fg{}}

Pour $n$ polynômes \og affines\fg{} $p_1, \cdots, p_n$ en $n-1$ variables $\ux = (x_1, \cdots, x_{n-1})$,
il y a une notion de \og bezoutien affine\fg{}, qui est un polynôme $\bez(\ux,\uy)$ de $\bfA[\ux,\uy]$,
défini sur le modèle suivant pour $n=3$
$$
\bez(\ux,\uy) =
\dfrac{
\begin {vmatrix}
p_1(x_1,x_2) &  p_2(x_1,x_2) &  p_3(x_1,x_2) \\
p_1(y_1,x_2) &  p_2(y_1,x_2) &  p_3(y_1,x_2) \\
p_1(y_1,y_2) &  p_2(y_1,y_2) &  p_3(y_1,y_2) \\
\end {vmatrix}
}
{(x_1-y_1)(x_2-y_2)}
$$
Il est clair que le déterminant est divisible par $(x_1-y_1)(x_2-y_2)$ car la spécialisation 
$y_1 := x_1$ fournit un déterminant ayant ses 2 premières lignes égales tandis que la
spécialisation $y_2 := x_2$ fournit un déterminant avec 2 dernières lignes égales.

\medskip

Afin de comparer cette notion affine et la notion homogène du début du
chapitre, on introduit pour un système \emph{homogène} $\uP =
(P_1, \dots, P_n)$ de $\bfA[\uX]$, un polynôme intermédiaire
$\Psi(\uX,\uY)$ à l'aide d'un déterminant d'ordre~$n$:
$$
\Psi(\uX,\uY) \ = \ 
\dfrac{1}{(X_1 - Y_1)\cdots(X_{n-1} - Y_{n-1})}
\left|
\begin{array}{ccc}
\cdots & P_j(X_1,X_2,\dots,X_n) & \cdots \\
\cdots & P_j(Y_1,X_2,\dots,X_n) & \cdots \\
& \vdots & \\
& P_j(Y_1,\dots, Y_{i-1}, X_i, \dots, X_n) &  \\
& \vdots & \\ 
\cdots & P_j(Y_1,\dots, Y_{n-1}, X_n) & \cdots \\
\end{array}
\right|
$$
C'est bien un polynôme car de la même manière, la spécialisation $Y_i
:= X_i$ conduit à un déterminant dont les lignes $i$ et $i+1$ sont
égales; en conséquence, le déterminant est divisible par les $(X_i-Y_i)_{1 \le
i \le n-1}$ donc par leur produit. Ce polynôme $\Psi$ ne dépend pas de $Y_n$
et sa spécialisation $X_n := 1$ fournit le bezoutien
affine~$\bez(\ux,\uy)$ des $n$ polynômes affines $p_i := (P_i)_{X_n :=
1}$
$$
\bez(\ux,\uy) = \Psi(x_1, \cdots, x_{n-1},1,\  y_1, \cdots, y_{n-1})
$$
On écrit ce polynôme dans $\bfA[\uX][\uY]$ sous la forme:
$$
\Psi(\uX,\uY)  = \sum_{\beta} \Psi_\beta(\uX)\,Y^\beta \qquad
\Psi_\beta(\uX) \in \bfA[\uX]
$$
Le coefficient $\Psi_\beta(\uX)$ en $Y^\beta$ de $\Psi$ est nul si $\beta_n \geqslant 1$
(car $\Psi$ ne dépend pas de $Y_n$).  En notant $\Bez = \Bez(\uX,\uY)$
le bezoutien rigidifié de $\uP$ défini en~\ref{DefBezoutien},
l'objectif est de comparer certains coefficients de $\Bez$ et de~$\Psi$,
coefficients au sens de $\bfA[\uX][\uY]$. Ce qui permettra, sous certaines
hypothèses, de faire le lien entre bezoutien homogène~$\Bez$ et bezoutien affine~$\bez$.

\begin{prop}
Pour tout $\beta \in \bbN^n$ tel que $|\beta| < \min(D)$, on a dans $\bfA[\uX]$:
$$
X_n^{\beta_n+1} \Bez_\beta(\uX) \,=\, X_n \Bez_{\widetilde\beta}(\uX) \,=\, \Psi_{\widetilde\beta}(\uX)
\qquad 
\text{où } 
\quad \widetilde\beta = (\beta_1, \ldots, \beta_{n-1}, 0)
$$
\end{prop}

\begin{proof}

Rappelons quelques informations sur la matrice bezoutienne rigidifiée $\dsV$
dont $\Bez$ est le déterminant. Sa dernière ligne est donnée par 
$$
\dsV_{n,j} \ =\ 
\frac{P_j(Y_1,\dots,Y_{n-1},X_n) - P_j(Y_1,\dots,Y_{n-1},Y_n)} {X_n-Y_n}
$$
Et ses $n-1$ premières lignes (constituées des $\dsV_{i,j}$ pour $i<n$)
ne dépendent pas de $Y_n$.

\noindent
On introduit le déterminant $\psi(T)$ d'ordre $n$
$$
\psi(T) \ = \ 
\left|
\begin {array}{ccccc}
\\
\noalign{\hbox{\hspace{1.4cm} $\dsV_{[1..n-1] \times [1..n]}$}} \\
\\
F_1(T) &\cdots & F_j(T) & \cdots & F_n(T) \\
\end {array}
\right|
\qquad \hbox {où} \qquad
F_j(T) \ =\ P_j(Y_1, \ldots, Y_{n-1}, T)
$$
Vu la dernière ligne de $\dsV$, on a
$$
(X_n-Y_n)\Bez \ =\ \psi(X_n) - \psi(Y_n)
$$
D'autre part, en effectuant sur le déterminant qui intervient dans la définition de~$\Psi$
les opérations élémentaires de lignes $\ell_i \leftarrow \ell_i-\ell_{i+1}$ pour $i =1,\dots,n-1$,
on obtient $\psi(X_n) = \Psi(\uX,\uY)$. Donc:
$$
X_n\Bez -Y_n\Bez \ =\ \Psi(\uX,\uY) - \psi(Y_n)
\leqno (\diamond)
$$
Dans cette égalité, on se concentre sur le coefficient en $Y^\beta$
des polynômes vus dans $\bfA[\uX][\uY]$.  Montrons qu'à droite,
le coefficient en $Y^\beta$ de $\psi(Y_n)$ est nul.
En effet, la dernière ligne du déterminant définissant~$\psi(Y_n)$ est
constituée des polynômes $F_j(Y_n) = P_j(\uY) \in \bfA[\uY]_{d_j}$ et
comme $|\beta|+1 \leqslant d_j$, le polynôme $F_j(Y_n)$ appartient à
l'idéal $\bfA[\uY]_{\geqslant |\beta|+1}$.  En développant le
déterminant $\psi(Y_n)$ par rapport à la dernière ligne, on obtient
l'appartenance $\psi(Y_n) \in \bfA[\uX][\uY]_{\geqslant
|\beta|+1}$ et ainsi le coefficient en $Y^\beta$ de $\psi(Y_n)$ est nul.

\medskip

Si $\beta_n \geqslant 1$, le coefficient en $Y^\beta$
de $\Psi(\uX,\uY)$ (qui ne dépend pas de~$Y_n$) est nul
et l'examen du coefficient en $Y^\beta$ de $(\diamond)$ fournit:
$$
X_n \Bez_\beta(\uX) - \Bez_{\beta'}(\uX) \,=\, 0
\quad \text{où}\quad \beta' = (\beta_1, \ldots, \beta_{n-1}, \beta_{n}-1) 
$$
Comme $|\beta'| < |\beta| < \min(D)$, on peut appliquer l'égalité ci-dessus à $\beta'$
à la place de $\beta$ si toutefois $\beta'_n \ge 1$. En itérant, on obtient:
$$
X_n^{\beta_n} \Bez_\beta(\uX) = \Bez_{\widetilde\beta}(\uX)
\quad \text{où}\quad \widetilde\beta = (\beta_1, \dots, \beta_{n-1}, 0)
$$
On a $|\widetilde\beta| < \min(D)$ et l'exament du coefficient en
$Y^{\widetilde\beta}$ de $(\diamond)$ donne cette fois
$X_n \Bez_{\widetilde\beta}(\uX) = \Psi_{\widetilde\beta}(\uX)$. D'où:
$$
X_n^{\beta_n+1} \Bez_\beta(\uX) \,=\, X_n \Bez_{\widetilde\beta}(\uX) \,=\, \Psi_{\widetilde\beta}(\uX)
$$
ce qui termine la démonstration.
\end{proof}

\medskip

En degré $d$, les opérations de déshomogénéisation et d'homogénisation entre $\uX = (X_1, \cdots, X_n)$
et $\ux= (x_1,\cdots,x_{n-1})$ founissent deux bijections réciproques l'une de l'autre:
$$
\bfA[\uX]_d \xrightarrow {X_i:=x_i\ i <n,\ X_n := 1} \bfA[\ux]_{\leqslant d},
\qquad\quad
\bfA[\ux]_{\leqslant d} \xrightarrow {f\mapsto X_n^df(X_1/X_n,\dots,X_{n-1}/X_n)} 
\bfA[\uX]_d 
$$
Idem en degré $d'$ entre $\uY = (Y_1, \cdots, Y_n)$ et $\uy = (y_1, \cdots, y_{n-1})$.

\medskip

La proposition précédente fournit en particulier $\Bez_\beta(X_n := 1)
= \Psi_{\widetilde\beta}(X_n := 1)$ et conduit au corollaire suivant
qui relie le bezoutien rigidifié homogène $\Bez$ de $\uP$ et le
bezoutien affine $\bez$ du système $X_n$-affinisé.

\begin {coro}

Soit $d$ tel que $d' := \delta-d$ vérifie $d'<\min(D)$.
Modulo les bijections de déshomogénéisation et d'homogénisation,
les deux applications ci-dessous s'identifient :
$$
\Bez_{d,d'}^\dagger :
\begin{array}[t]{rcl}
\bfA[\uY]_{d'} & \longrightarrow & \bfA[\uX]_d \\ [0.1cm]
Y^\beta & \longmapsto & \Bez_\beta(\uX) = \coeff_{Y^\beta}(\Bez)
\end{array}
\qquad \text{et } \qquad 
\begin{array}[t]{rcl}
\bfA[\uy]_{\leqslant d'} & \longrightarrow & \bfA[\ux]_{\leqslant d} \\ [0.1cm]
y^{\beta'} & \longmapsto & \bez_{\beta'}(\ux) = \coeff_{y^{\beta'}}(\bez)
\end{array}
$$
Nous avons commis une entorse au réglement  en définissant $\Bez_{d,d'}^\dagger$
sur $\bfA[\uY]_{d'}$ et non sur son dual $\bfA[\uY]_{d'}^\star$.
\end {coro}

\medskip

Soient $n$ polynômes \og affines \fg{} 
$p_1, \dots, p_n$ en $n-1$ variables $x_1, \dots, x_{n-1}$ 
et $D = (d_1, \dots, d_n)$ vérifiant $\deg p_i \leqslant d_i$.
On définit
$P_i = X_n^{d_i}p_i(\frac{X_1}{X_n}, \dots, \frac{X_{n-1}}{X_n})$, 
polynôme homogène de degré $d_i$ de $\bfA[\uX] = \bfA[X_1, \dots, X_n]$ 
et on pose 
$$
\Res_D(p_1, \dots, p_n) \ = \ 
\Res(P_1, \dots, P_n)
$$
Voici une petite application pour $n=2$ montrant que nous pouvons
être déraisonnables en prenant un canon pour tuer une mouche.

\begin{prop}

Soit $p,q$ deux polynômes {\rm affines} (en une seule variable) de degré $\leqslant e$.
Notons~$\bez$ leur bezoutien affine en deux variables $x,y$ 
$$
\bez(x,y) \ = \ 
\dfrac{1}{x-y}
\begin{vmatrix}
p(x) & q(x) \\
p(y) & q(y) \\
\end{vmatrix}
\ = \ 
\dfrac{p(x)q(y) - p(y)q(x)}{x-y}
$$
C'est un polynôme symétrique de degré $\le d := e-1$ que l'on écrit:
$$
\bez(x,y) = \sum_{0\le i,j \le d} a_{i,j}\,x^i y^j  \ = \ 
\sum_{j = 0}^d \bez_j(x)\,y^{j}  \quad \text{où} \quad
\bez_j(x) = \sum_{i = 0}^{d} a_{i,j}x^i
$$
On lui associe l'endomorphisme $\bezendo(p,q)$ de $\bfA[\ux]_{\le d}$ défini par $x^j \mapsto \bez_{d-j}$.
Alors $\Res_{(e,e)}(p,q)$ est donné par un déterminant d'ordre $e$:
$$
\Res_{(e,e)}(p,q) = \det\big(\bezendo(p,q)\big) =
\det(a_{i,d-j})_{0\leqslant i,j\leqslant d} = \det(a_{d-i,j})_{0\leqslant i,j\leqslant d}
$$
\end{prop}

\label{NOTA17-bezendo}%
\index{matrice!de Bezout}%

\begin{proof}

On va appliquer le corollaire précédent à
$$
P(X_1, X_2) = X_2^e\, p(X_1/X_2), \quad Q(X_1, X_2) = X_2^e\, q(X_1/X_2), \qquad
D = (e,e), \qquad \delta = 2(e-1)
$$
Vérifions l'on peut utiliser la méthode de Sylvester (cf le
chapitre \ref{ChapSylvesterHybride}) en degré $d=e-1$ et qu'elle
permet d'écrire (sans dénominateur):
$$
\Res(P,Q)  = \det\swbsOmega_d(\Bez_{d,d'})
$$
Effectivement, avec $d=e-1$ on a $d'=d < \min(D)$, ce qui permet
d'utiliser d'une part la méthode de Sylvester et d'autre part le
corollaire précédent. Celui-ci affirme que pour $0 \leqslant
j \leqslant d$, la colonne qui exprime $\Bez_{(j, d-j)}(\uX)$ dans
$\bfA[\uX]_d$ est égale à la colonne qui exprime $\bez_{j}(x)$ dans
$\bfA[x]_{\leqslant d}$.

\medskip

Le fait que $d=d' < \min(D)$ a comme conséquence que $\Jex_{1,d} = 0$, a fortiori $\Jex_{2,d}=0$
donc il n'y a pas de dénominateur. Comme $\Smac_{0,d} = \bfA[\uX]_d$,
il ne reste dans la matrice de $\swbsOmega_d(\Bez_{d,d'})$ que le bloc~$\Bez_{d,d'}^\dagger$. 
Par définition, puisque $\emouton - (j,d-j) = (d-j,j)$:
$$
\swbsOmega_d(\Bez_{d,d'}) : \bfA[\uY]_{d'} \to \bfA[\uX]_d
\qquad
Y^\beta = Y_1^jY_2^{d-j} \mapsto \coeff_{Y^{\emouton-\beta}}(\Bez) = \coeff_{Y_1^{d-j}Y_2^j}(\Bez)
$$
L'application ci-dessus s'identifie à l'application linéaire définie sur la base des monômes:
$$
\bfA[y]_{\le d'} \to \bfA[x]_{\le d}, \qquad
y^j \mapsto  \coeff_{d-j}(\bez) \overset{\rm def.}{=} \bez_{d-j} \overset{\rm def.}{=} \sum_{i=0}^d a_{i,d-j}\, x^i
$$
Puisque $d'=d$, on peut identifier la base monomiale de l'espace de
départ et celle de l'espace d'arrivée, assimilant ainsi cette
application linéaire à l'endomorphisme $\bezendo(p,q)$ de
$\bfA[x]_{\le d}$. Son déterminant est égal à
$\det\swbsOmega_d(\Bez_{d,d'})$.  La matrice $(b_{i,j})_{0 \leqslant
i,j\leqslant d}$ de $\bezendo(p,q)$ dans la base $(x^0, \dots, x^d)$
est donnée $b_{i,j} = a_{i,d-j}$ tandis que celle dans la base
$(x^d, \dots, x^0)$ est donnée par $b'_{i,j} = a_{d-i,j}$.

\end{proof}

\begin{rmq} [A propos des deux matrices de Bezout, transposées l'une de l'autre]

Les deux matrices qui interviennent dans la proposition précédente:
$$
B = (a_{i,d-j})_{0 \leqslant i,j\leqslant d}, \qquad\qquad
B' = (a_{d-i,j})_{0 \leqslant i,j\leqslant d}
$$
sont attribuées à Bezout. Elles sont symétriques \emph{par rapport à l'anti-diagonale}
et transposées l'une de l'autre puisque $b_{j,i} = a_{j,d-i} = a_{d-i,j} = b'_{i,j}$.
Comme $B$ (resp. $B'$) est la matrice de $\bezendo(p,q)$ dans la base $(x^0, \dots, x^d)$
(resp. $(x^d, \dots, x^0)$):
$$
[x^0, \dots, x^d]\,B = [\bez_d, \dots, \bez_0], \qquad
[x^d, \cdots, x^0]\,B' = [\bez_0, \dots, \bez_d]
$$
La plupart des auteurs considèrent la matrice
symétrique $A := (a_{i,j})_{0 \leqslant i,j\leqslant d}$ dont le
déterminant n'est pas normalisé pour le jeu étalon. En effet,
en notant $J_e$ la matrice de l'involution renversante de $\fS_e$:
$$
A = B\,J_e = J_e\,B'   \qquad\qquad
J_e = \begin {bmatrix}
       &     &1  \\
       &1           \\
1      &     &  \\
\end {bmatrix}
$$
donc
$$
\det A = (-1)^{\frac{e(e-1)}2} \times \det \bezendo(p,q)
$$
Note: à partir des relations $B = AJ_e$ et $B'= J_eA$, on \og retrouve\fg{}  $\transpose{B} = B' = J_e B J_e$.
Le fait que $J_e$ \og renverse les vecteurs\fg{} confirme l'équivalence:
$$
[x^0, \dots, x^d]\,B = [\bez_d, \dots, \bez_0] \quad \Longleftrightarrow\quad
[x^d, \cdots, x^0]\,B' = [\bez_0, \dots, \bez_d]
$$
\end{rmq}

\subsubsection*{\og Transformation\fg{} de la matrice de Sylvester de $(P,Q)$
en ``la'' matrice de Bezout de $(p,q)$}

On veut fournir une preuve directe de l'égalité $\Res(P,Q) = \det\bezendo(p,q)$ en
procédant à de \og l'élimination dans la matrice de Sylvester\fg{}.
Commençons par quelques rappels.
Soient $P,Q$ deux polynômes homogènes en deux variables $X_1,X_2$, de
degrés respectifs $d_P, d_Q$, de degré critique $\delta = d_P-1 + d_Q-1$.
L'endomorphisme $W_{1,\delta+1}(P,Q)$ opère sur $\bfA[X_1,X_2]_{\delta+1}$ donc
en dimension $\delta+2 = d_P + d_Q$. Et, à condition de considérer l'application
de Sylvester dans des bases ad-hoc: 
$$
W_{1,\delta+1}(P,Q) = \Syl_{\delta+1}(P,Q)
$$
Prenons désormais $d_P=d_Q=e$ de sorte que $\delta+2 = 2e$.
A $p(x) = \sum_{i=0}^e p_ix^i$, en continuant à noter $e=d-1$,
associons deux matrices triangulaires de dimension~$e$,
symétriques par rapport à l'anti-diagonale:
$$
T(p) = (p_{j-i})_{0 \leqslant i,j\leqslant d}, \qquad\qquad
T'(p) = (p_{e+j-i})_{0 \leqslant i,j\leqslant d}
$$
qui se visualisent ainsi:
$$
T(p) =
\begin {bmatrix}
p_0 &p_1 &\cdots  &p_{e-2} &p_{e-1} \\
    &p_0 &\cdots  &       &p_{e-2} \\
    &    &\ddots  &       &\vdots  \\
    &    &        &p_0    &p_1 \\
    &    &        &       &p_0 \\
\end {bmatrix}
\qquad\qquad
T'(p) =
\begin {bmatrix}
p_e \\
p_{e-1}  &p_e \\
\vdots  &     &\ddots   \\
p_2     &     &        &p_e    \\
p_1     &p_2  &\cdots  &p_{e-1} &p_e \\
\end {bmatrix}
$$
Les relations suivantes de commutation vont se révéler fondamentales:
$$
T(p)T(q) = T(q)T(p), \qquad\qquad T'(p)T'(q)=T'(q)T'(p)
$$
Elles s'obtiennent sans calcul en considérant la $\bfA$-algèbre
commutative $\bfA[x]/\langle x^e\rangle$ libre de dimension $e$, munie
de deux bases naturelles $(x^d, \dots, x^0)$ et $(x^0, \dots, x^d)$ où
$d=e-1$.  La matrice $T(p)$ est celle de la multiplication par
$\sum_{i=0}^d p_ix^i$ dans la première base tandis que $T'(p)$ est celle
de la multiplication par $\sum_{i=0}^d p_{e-i}x^i$ dans la seconde.

\medskip
En fixant comme base de $\bfA[X_1,X_2]_{2e-1}$:
$$
(X_1^{2e-1}, X_1^{2e-2}X_2, \dots,  X_1X_2^{2e-2}, X_2^{2e-1})
$$
on dispose d'une structure en 4 blocs de dimension $e$ pour $W_{1,\delta+1}$:
$$
W_{1,\delta+1}(P,Q) = \begin {bmatrix}
T'(p) & T'(q) \\
T(p)  & T(q)  \\
\end {bmatrix}
$$
Or, en présence de 4 matrices carrées $A,B,C,D$ de même dimension $e$
telles que $C,D$ commutent, il est bien connu que
$$
\det \begin {bmatrix} A&B\\ C&D\\ \end{bmatrix} = \det(AD-BC)
$$
Une manière de le montrer consiste à écrire:
$$
\begin {bmatrix} A&B\\ C&D\\ \end{bmatrix} \begin {bmatrix} D&0_e\\ -C&\Id_e\\ \end{bmatrix} =
\begin {bmatrix} AD-BC&B\\ 0_e&D\\ \end{bmatrix}
\qquad \text{d'où} \qquad
\det \begin {bmatrix} A&B\\ C&D\\ \end{bmatrix} \ \det D = \det(AD-BC) \ \det D
$$
En remplaçant $D$ par $D + x\Id_e$ (qui commute encore à $C$), on peut simplifier par $\det(D + x\Id_e)$
qui est régulier puis faire $x := 0$ pour obtenir le résultat qualifié de bien connu.

\medskip

En l'appliquant à $W_{1,\delta+1}(P,Q)$, on obtient:
$$
\det W_{1,\delta+1}(P,Q) =  \det\big(T'(p)T(q) - T'(q)T(p)\big)
$$
On remplace ainsi le déterminant d'ordre $2e$ à gauche par le
déterminant d'ordre $e$ à droite.  Après tout, on \emph {pourrait}
prendre comme définition de matrice de Bezout de $(p,q)$ la matrice de
dimension~$e$ à droite.  Mais on va montrer que cette matrice n'est
autre que la matrice $B' = (a_{d-i,j})_{0 \leqslant i,j\leqslant d}$
de l'endomorphisme~$\bezendo(p,q)$ (cf. la remarque précédente):
$$
T'(p)T(q) - T'(q)T(p) = B'
\qquad \text{d'où} \qquad
\boxed{\det W_{1,\delta+1}(P,Q) = \det \bezendo(p,q)}.
$$
Avant de fournir la preuve de l'égalité matricielle, une illustration
en considérant le cas $e=3$:
$$
W_{1,\delta+1} =
\begin {bmatrix}
p_3 &.   &.   &q_3 &.   &.  \\
p_2 &p_3 &.   &q_2 &q_3 &.  \\
p_1 &p_2 &p_3 &q_1 &q_2 &q_3 \\
p_0 &p_1 &p_2 &q_0 &q_1 &q_2 \\
.   &p_0 &p_1 &.   &q_0 &q_1 \\
.   &.   &p_0 &.   &.   &q_0 \\
\end {bmatrix}
$$
Notons $M$ la matrice définie par:
$$
M =
\begin {bmatrix}
p_3 &.   &.   \\
p_2 &p_3 &.   \\
p_1 &p_2 &p_3 \\
\end {bmatrix}
\begin {bmatrix}
q_0 &q_1 &q_2 \\
.   &q_0 &q_1 \\
.   &.   &q_0 \\
\end {bmatrix}
-
\begin {bmatrix}
q_3 &.   &.   \\
q_2 &q_3 &.   \\
q_1 &q_2 &q_3 \\
\end {bmatrix}
\begin {bmatrix}
p_0 &p_1 &p_2 \\
.   &p_0 &p_1 \\
.   &.   &p_0 \\
\end {bmatrix}
$$
Voici les $\bez_j$:
$$
\begin {array} {ccl}
\bez_0 &=& (-p_0q_3 + p_3q_0)x^2 + (-p_0q_2 + p_2q_0)x - p_0q_1 + p_1q_0
\\
\bez_1 &=& (-p_1q_3 + p_3q_1)x^2 + (-p_0q_3 - p_1q_2 + p_2q_1 + p_3q_0)x - p_0q_2 + p_2q_0
\\
\bez_2 &=& (-p_2q_3 + p_3q_2)x^2 + (-p_1q_3 + p_3q_1)x - p_0q_3 + p_3q_0
\\
\end {array}
$$
A l'aide des $\bez_j$, comment affecter les colonnes de $B'$,
matrice de l'endomorphisme $\bezendo(p,q)$ dans la base $(x^2, x^1, x^0)$?
Est ce cohérent avec l'égalité $M = B'$?

\subsubsection*{Preuve des égalités $B' = T'(p)T(q)- T'(q)T(p) = T(q)T'(p) - T(p)T'(q)$
où $B'=(a_{d-i,j})_{0 \le i,j \le d}$}

$\rhd$
Introduisons l'opérateur de décalage descendant $\tau : \bfA[x] \to \bfA[x]$ défini par
$$
\tau(x^k) = \begin {cases} x^{k-1} &\text{si $k\ge 1$}\\ 0 &\text{sinon} \\ \end{cases}
$$
et montrons pour $0 \le j \le d$:
$$
\bez_j = \tau(p)q_j + \cdots + \tau^{j+1}(p)q_0 -
\big(\tau(q)p_j + \cdots + \tau^{j+1}(q)p_0\big)
$$
A cet effet, on écrit
$$
\frac{p(x)q(y) - p(y)q(x)}{x-y} =
\frac{p(x) - p(y)}{x-y}\,q(y) - \frac{q(x) - q(y)}{x-y}\,p(y)
$$
En utilisant
$$
\frac{p(x) - p(y)}{x-y} = \sum_{i=0}^e p_i \sum_{k=0}^{i-1} x^{i-k-1} y^k =
\sum_{k=0}^d \biggl(\,\sum_{i=k+1}^e p_ix^{i-k-1}\biggr) y^k =
\sum_{k=0}^d \tau^{k+1}(p) y^k
$$
et l'analogue pour $q$, il vient
$$
\frac{p(x)q(y) - p(y)q(x)}{x-y} =
\sum_{k=0}^d \big(\tau^{k+1}(p)q(y) - \tau^{k+1}(q)p(y)\big) y^k
$$
L'égalité convoitée pour $\bez_j$ s'obtient par examen du coefficient de $y^j$ dans la somme de droite.

\medskip
$\rhd$
Remarquons que $\tau$ intervient également dans la définition de $T'$ de la manière suivante
$$
[x^d,\dots,x^0]\,T'(p) = [\tau(p), \dots, \tau^e(p)], \qquad\qquad
[x^d,\dots,x^0]\,T'(q) = [\tau(q), \dots, \tau^e(q)]
$$
Désignons par $T_j(p)$ la colonne $j$ de $T(p)$ (numérotation à partir de $0$)
i.e. $T_j(p) = \transpose{[p_j, \dots, p_0,0,\dots,0]}$.
Idem avec $T_j(q) = \transpose{[q_j, \dots, q_0,0,\dots,0]}$.
L'égalité fournie pour $\bez_j$ peut s'écrire:
$$
\bez_j = [\tau(p), \dots, \tau^e(p)].T_j(q) - [\tau(q), \dots, \tau^e(q)].T_j(p)
$$
ou encore
$$
\bez_j = [x^d,\dots,x^0]\,\big(T'(p).T_j(q) - T'(q).T_j(p) \big)
$$
On a donc:
$$
[x^d, \cdots, x^0]\,\big(T'(p)T(q) - T'(q)T(p)\big)  = [\bez_0, \dots, \bez_d]
$$
Or $B'$ vérifie la même égalité
$$
[x^d, \cdots, x^0]\,B' = [\bez_0, \dots, \bez_d]
$$
donc $B' = T'(p)T(q) - T'(q)T(p)$.

\medskip
La seconde égalité se démontre à partir de la première en remarquant que les 5 matrices
qui interviennent sont symétriques par rapport à l'anti-diagonale.
Pour une matrice carrée $M$ de dimension $e$, posons $\theta(M) = J_e\transpose{M}J_e$.
Alors $\theta$ est additive, anti-multiplicative, $\theta(M_1M_2) =\theta(M_2)\theta(M_1)$, et
$M$ est symétrique par rapport à l'anti-diagonale si et seulement si $\theta(M) = M$.
En appliquant $\theta$ à $B' = T'(p)T(q) - T'(q)T(p)$, on obtient
$B' = T(q)T'(p) - T(p)T'(q)$.

\subsubsection*{Un calcul plus synthétique pour montrer $J_eA = T'(p)T(q) - T'(q)T(p)$}

Il consiste à associer à tout $F \in \bfA[x,y]/\langle x^e,y^e\rangle$ une matrice carrée
$\matrice(F) = (F_{i,j})_{0 \le i,j \le d}$ de dimension $e$ où
$$
F = \sum_{0 \le i,j \le d} F_{i,j}x^i y^j
$$
On vérifie facilement que:
$$
\matrice\big(F.q(y)\big) = \matrice(F)\,T(q), \qquad
J_e\,\matrice\left(\frac{p(x)-p(y)}{x-y}\right) = T'(p)
$$
On en déduit
$$
J_e\,\matrice\left(\frac{p(x)-p(y)}{x-y}.q(y)\right) = T'(p)T(q)
$$
et on termine en utilisant 
$$
\frac{p(x)q(y) - p(y)q(x)}{x-y} =
\frac{p(x) - p(y)}{x-y}\,q(y) - \frac{q(x) - q(y)}{x-y}\,p(y)
$$

\cleardoublepage

\section{Séries de Hilbert-Poincaré}
\label{ChapSeries}

Expliquons de façon informelle comment lister tous les monômes $X^\alpha 
= X_1^{\alpha_1} \cdots X_n^{\alpha_n}$ \textit{de manière efficace, condensée}.
On peut par exemple utiliser l'outil \og série formelle \fg{}.

En considérant le produit des $n$ séries :
$$
(1+X_1 + \cdots + X_1^{\alpha_1} + \cdots)
\ \cdots \ 
(1+X_n + \cdots + X_n^{\alpha_n} + \cdots)
$$
et en développant (il y a donc une infinité de termes), 
chaque monôme $X^\alpha$ de $\bfA[\uX]$ apparaît une et une seule fois !
Ceci est complètement banal. Ce qui l'est moins, c'est que 
ce produit de séries est une \textit{fraction rationnelle}, 
à savoir $\dfrac{1}{(1-X_1) \cdots (1-X_n)}$.

\bigskip

Expliquons, toujours à l'aide des séries, comment obtenir 
tous les monômes vérifiant une condition donnée relative au format $D$.
Par exemple, listons tous les monômes 
$X^\alpha = X_1^\alpha \cdots X_n^{\alpha_n}$ vérifiant $\alpha_j < d_j$
\textbf{pour un $j$ fixé}.
Considérons pour cela le produit 
(le $j$\up{ème} facteur est un \textit{polynôme}, les autres facteurs 
sont de \og vraies séries \fg{})
$$
(1+X_1 + \cdots + X_1^{\alpha_1} + \cdots)
\quad \cdots \quad
(1+X_j + \cdots + X_j^{\alpha_j} + \cdots + X_j^{d_j-1})
\quad \cdots \quad
(1+X_n + \cdots + X_n^{\alpha_n} + \cdots)
$$
produit qui, après développement, 
est égal à la somme des monômes $X_1^{\alpha_1} \cdots X_j^{\alpha_j} \cdots X_n^{\alpha_n}$ vérifiant  
la condition $\alpha_j < d_j$.
Ce qui est remarquable, c'est que ce produit de \textit{séries} est la \textit{fraction rationnelle}
$\dfrac{1-X_j^{d_j}}{(1-X_1) \cdots (1-X_n)}$.

\bigskip

Voici une deuxième passe, toujours très simple.
Cherchons maintenant à lister tous les monômes 
$X^\alpha = X_1^\alpha \cdots X_n^{\alpha_n}$ vérifiant $\alpha_i \geqslant d_i$
\textbf{pour un $i$ fixé}.
Il suffit de considérer le produit des $n$ séries suivantes :
$$
(1+X_1 + \cdots + X_1^{\gamma_1} + \cdots)
\quad \cdots \quad
X_i^{d_i}(1+X_j + \cdots + X_i^{\gamma_i} + \cdots ) 
\quad \cdots \quad
(1+X_n + \cdots + X_n^{\gamma_n} + \cdots)
$$
qui est égal à la fraction rationnelle 
$\dfrac{X_i^{d_i}}{(1-X_1) \cdots (1-X_n)}$.

\bigskip

Fixons un entier $k \ge 0$. Imaginons que nous voulions maintenant énumérer tous 
les monômes \textit{extérieurs} de $\Mmac_k$ à l'aide 
de vrais monômes
via la correspondance $X^\beta e_J \leftrightarrow X^{\beta + D(J)}$ déjà
intervenue (cf. le pont $\psi_J$ dans la proposition \ref{JDecompositionEndo}
et la remarque qui suit cette proposition).

Nous expliquerons plus loin (cf. page~\pageref{Series:graduationKoszul}) 
la raison de ce choix.

On peut commencer par \og faire des paquets \fg{} 
en utilisant la décomposition:
$$
\Mmac_k \ = \ \bigoplus_{i,J}  \Mmac_k^{(i)}[J]
$$
où pour $i \in \{1,\dots, n\}$ et $J \subset \{1,\dots, n\}$ de cardinal $k$, 
le module $\Mmac_k^{(i)}[J]$ désigne le $\bfA$-module de base les monômes extérieurs 
$X^\beta e_J \in \Mmac_k$ vérifiant $\minDiv(X^\beta) = i$ et 
\begin{center}
$J \subset \{i+1, \dots, n\}$ 
\qquad ou encore \qquad 
$i < \min J$
\end{center}
Au passage, remarquons que l'on peut paramétrer ces monômes extérieurs par le couple $(i,J)$ 
ou bien par le couple $(J,i)$. La première option exprime $J$ en fonction de $i$ 
et permet d'avoir des informations sur le module 
$\Mmac_k^{(i)}$, alors que la deuxième permet de récupérer le module $\Mmac_k[J]$ ;
ceci via les décompositions :
$$
\Mmac_k^{(i)} \ = \ \bigoplus_{J\subset\{i+1..n\}\atop\#J=k}  \Mmac_k^{(i)}[J]
\qquad \qquad
\Mmac_k[J] \ = \ \bigoplus_{i < \min J}  \Mmac_k^{(i)}[J]
$$

\noindent
Fixons $i$ et $J$ et regardons d'abord le \og vrai monôme \fg{}
$X^\beta$. La condition $\minDiv(X^\beta) = i$ se traduit par les
conditions :
$$
\forall\, j< i, \ \beta_j < d_j, \qquad \text{ et } \qquad \beta_i \geqslant d_i
$$
Notons $S_{i,J}$ la somme $\sum X^\beta$ de tels monômes.
Le lecteur constatera que l'on a l'égalité suivante 
(faisant intervenir le produit de $n$ séries, qui sont des fractions rationnelles) :
$$
S_{i,J} \ = \ 
\prod\limits_{j < i} 
\dfrac{1-X_j^{d_j}}{1-X_j} \times 
\dfrac{X_i^{d_i}}{1-X_i} \times 
\prod\limits_{k > i} 
\dfrac{1}{1-X_k}
$$
c'est-à-dire 
$$
S_{i,J} \ = \ 
\dfrac{N_{i,J}}{(1-X_1) \cdots (1-X_n)} 
\qquad \text{avec} \qquad 
N_{i,J}
\ = \ 
\prod\limits_{j < i} (1-X_j^{d_j}) \times X_i^{d_i}
\ = \ 
\overline{X}^{D([1...i[)} X_i^{d_i}
$$
On a utilisé la notation suivante :
$\overline{X}^{D([1...i[)} = \prod\limits_{j < i} (1-X_j^{d_j})$, 
notation qui reviendra deux autres fois sous cette forme 
dans les lignes à venir et sous des formes voisines que 
le lecteur pourra, s'il le souhaite, deviner par lui-même, 
ou s'il préfère, consulter la page~\pageref{Series:graduationKoszul}.

\label {NOTA18-XbarDI}%

\noindent

Revenons maintenant aux \textit{monômes extérieurs} $X^\beta e_J$ 
de $\Mmac_k^{(i)}[J]$ en utilisant la correspondance précisée auparavant,
qui réalise $e_J \leftrightarrow X^{D(J)}$.
On obtient alors l'égalité :
$$
\Serie\big(\Mmac_k^{(i)}[J]\big)
\ = \ 
\dfrac{N'_{i,J}}{(1-X_1) \cdots (1-X_n)} 
\qquad \text{avec} \qquad 
N'_{i,J}
\ = \
N_{i,J} X^{D(J)}\overset{\rm def.}{=}
\overline{X}^{D([1...i[)} X_i^{d_i}X^{D(J)}
$$
Enfin, pour obtenir la série initialement convoitée, 
on somme sur $i \in \{1,\dots, n\}$ et $J \in \calP_k(\{1,\dots, n\})$ 
avec la contrainte 
$J \subset \{i+1,\dots, n\}$ (ou encore avec la contrainte $i < \min J$):
$$
\Serie(\Mmac_k) 
\ = \ 
\dfrac{N_k}{(1-X_1) \cdots (1-X_n)} 
\qquad \text{où} \qquad 
N_k
\ = \ 
\displaystyle
\sum_{i,J \atop i < \min J}
\overline{X}^{D([1...i[)} X_i^{d_i}X^{D(J)}  
$$
Ce paramétrage par $(i,J)$ peut être simplifié en remarquant que l'application
$$
\begin{array}[t]{ccc}
\big\{(i,J) \in \{1,\dots, n\} \times \calP_k(\{1,\dots, n\}) 
\mid i< \min J\big\} & \longrightarrow & \calP_{k+1}(\{1,\dots, n\})  \\ [0.3cm]
(i,J) & \longmapsto &  I := i\vee J \\
\end{array}
$$
est bijective, de réciproque $(i,J) := (\min I, I\setminus\min I)
\longleftarrow I$, d'où la ré-écriture du numérateur $N_k$
$$
N_k
\ = \ 
\sum_{\#I \, = \, k+1} \ 
 \overline{X}^{D([1...\min I[)} X^{D(I)} 
$$
Le lecteur particulièrement attentif aura remarqué que 
cette bijection est celle qui porte l'isomorphisme $\varphi : 
\Smac_{k+1} \rightarrow \Mmac_k$ (cf.~\ref{varphiIso}).
Comme celui-ci est gradué de degré $0$, on obtient à peu de frais la série 
de~$\Smac_{k+1}$, écrite ci-dessous en ayant posé $k' = k+1$, entier $\ge 1$:
$$
\Serie(\Smac_{k'})  
\ = \ 
\dfrac{N'(k')}{(1-X_1) \cdots (1-X_n)} 
\qquad \qquad 
N'(k')
\ = \ 
\sum_{\#I \, = \, k'} \ 
 \overline{X}^{D([1...\min I[)} X^{D(I)} 
$$
Cette formule, établie pour $k' \ge 1$, est-elle valide pour $k'=0$?
La réponse est oui car nous avons l'égalité:
$$
\Serie(\Smac_0)  
\ = \ 
\dfrac{N'_0}{(1-X_1) \cdots (1-X_n)}
\qquad \text{où} \qquad
N'_0 = \prod\limits_{j=1}^n (1-X_j^{d_j})
$$
Le lecteur pourra la vérifier mais nous l'informons 
que les séries des modules $\Smac_\sbullet$ seront revisitées en~\ref{SerieSk}.
Grâce à la convention $\min \emptyset = n+1$, le
numérateur $N'(0)$ vaut $\overline{X}^{D([1...n+1[)}X^{D(\emptyset)}$,
ce qui coïncide avec~$N'_0$.

Une fois ce \og petit mécanisme maîtrisé \fg{}, il n'est pas difficile 
de trouver \og de tête \fg{} toutes les séries de nos modules. 
C'est pour cette raison que nous ne donnerons pas toutes les preuves des résultats suivants.
Mais avant cela, posons le cadre formel 
et définissons proprement ce que l'on entend par 
\textit{série} d'un module gradué.

Lorsque l'on parle de $\bfA$-module gradué sur un anneau $\bfA$ a priori non gradué,
il est sous-entendu que~$\bfA$ est muni de la graduation triviale 
(\idest{} $\bfA_0 = \bfA$ où $0$ est le zéro du monoïde de la graduation, les
autres composantes homogènes étant nulles).
Il va justement être question ici de fine graduation (ou $\bbN^n$-graduation) 
pour des $\bfA$-modules~$E$, l'anneau $\bfA$ étant quelconque. Une structure
de $\bfA$-module $\bbN^n$-gradué est la donnée d'une somme directe 
$E = \bigoplus_{\gamma \in \bbN^n} E_\gamma$ 
où chaque $E_\gamma$ est un $\bfA$-module.

Il va être commode de disposer d'une version multiplicative du monoïde~$\bbN^n$ 
en considérant le monoïde multiplicatif engendré par $\uT = (T_1, \ldots, T_n)$ :
$$ 
(0, \dots, \underset{i}{1}, \dots, 0) \in \bbN^n 
\quad \leftrightarrow \quad
T_i \in \bbZ[T_1, \dots, T_n]
$$
Faire intervenir $\bbZ[\uT]$ permet d'associer à $E$ sa série
d'Hilbert-Poincaré fine-graduée définie par :
$$
\Serie(E) \ = \
\sum_{\gamma \in \bbN^n} \dim_{\bfA}(E_\gamma) \, T^\gamma
\ \in\, \bbZ[[\uT]]
$$
étant sous-entendu que chaque $E_\gamma$ est un $\bfA$-module libre de 
dimension finie.

\noindent
Notons qu'une structure de $\bfA[\uX]$-module $\bbN^n$-gradué sur $E$ fournit une
structure de $\bfA$-module $\bbN^n$-gradué (prendre les mêmes composantes).
\`A cette occasion, on procédera à l'abus de notation $X_i \leftrightarrow T_i$.
Si par exemple $E = \bfA[X,Y]$, on écrira:
$$
\Serie\big(\bfA[X,Y]\big) \ = \  1 + X + Y + X^2 + XY + Y^2 + \cdots \ = \  \dfrac{1}{(1-X)(1-Y)}
$$
plutôt que 
$$
\Serie\big(\bfA[X,Y]\big) \ = \  
1 + T_1 + T_2 + T_1^2 + T_1T_2 + T_2^2 + \cdots \ = \  \dfrac{1}{(1-T_1)(1-T_2)}
$$\noindent
\'Etant donné un $\bfA$-module $\bbN^n$-gradué $E$, on peut également
considérer la raw-graduation (ou $\bbN$-graduation) sous-jacente 
$E = \bigoplus_{d \in \bbN} E_d$ où $E_d$ est défini par 
$E_d = \bigoplus_{|\gamma| = d} E_\gamma$. 
Sa série d'Hilbert-Poincaré raw-graduée est la suivante :
$$
\Serie(E) \ = \ 
\sum_{d \in \bbN} \dim_{\bfA}(E_d) \, t^d
\ \in\, \bbZ[[t]]
$$
Elle est obtenue en réalisant $T_i := t$ dans la série fine-graduée de $E$.
Le contexte indiquera clairement de quel type de série on parle.

Dans cette section, $\bfA$ est un anneau quelconque et on fixe un format de degrés
$D = (d_1, \dots, d_n)$.
Pour une partie $I$ de $\llbracket 1,n\rrbracket$, on notera $D(I)$ 
le $n$-uplet dont la $i$-ème composante vaut $d_i$ si $i \in I$ et $0$ sinon.
On utilisera les notations 
$$
X^{D(I)} \ = \ \prod_{i \in I} X_i^{d_i}
\qquad \text{et} \qquad 
\overline{X}^{D(I)} \ = \ \prod_{i \in I} (1 - X_i^{d_i})
$$
\label{NOTA18-XDI}%
La fine graduation sur $\bfA[\uX]$ s'étend naturellement à $\bigwedge(\bfA[\uX]^n)$ en 
posant $\deg(e_I) = X^{D(I)}$. Notons que cette graduation est celle qui rend 
le complexe de Koszul du jeu étalon $\uX^D = (X_1^{d_1}, \dots, X_n^{d_n})$ fine-gradué.
Ainsi, pour $\gamma \in \bbN^n$, on a 
$$
\rmK_{k,\gamma} \ = \ \bigoplus_{\alpha, I} \bfA\, X^\alpha e_I 
\quad \text{avec $\alpha + D(I)  = \gamma$ \ et \ $\#I= k$}
$$
\label{Series:graduationKoszul}
Dit autrement, $\deg(X^\alpha e_I) = \alpha + D(I) \in \bbN^n$.
Pour la raw-graduation, on a $\deg(X^\alpha e_I) = |\alpha | + d_I \in \bbN$.

\medskip

Il n'est pas inutile de préciser, même si c'est un peu formel, que 
\og réaliser $T_i := t$ \fg{} est une manière de parler du morphisme de monoïdes
$$
\bbN^n \to \bbN  , \qquad \alpha \mapsto |\alpha| \overset{\rm def}{=} \sum_i \alpha_i
$$
Pour s'en convaincre, il suffit de contempler le diagramme de morphismes
de monoïdes:
$$
\xymatrix @C=2.5cm{
\bbN^n \ar@{<->}[r]_-{\simeq}^-{\textstyle \alpha\ \leftrightarrow\ T^\alpha}
    \ar[d]         &T_1^\bbN \cdots T_n^\bbN \ar[d]  \\
\bbN \ar@{<->}[r]_-{\simeq}^-{\textstyle d\ \leftrightarrow\ t^d}
                   &t^\bbN   \\
}
$$
L'isomorphisme horizontal en haut réalise  $\varepsilon_i \leftrightarrow T_i$.
La spécialisation $(T_i := t)_{\forall i}$ est le morphisme vertical à droite qui
envoie $T^\alpha$ sur $t^{|\alpha|}$. Si bien que le morphisme vertical à gauche qui
lui correspond est bien $\alpha \mapsto |\alpha|$. On pourra remarquer que
\begin{center}
$T^{D(I)}$, qui s'identifie à $\sum\limits_{i \in I}d_i\varepsilon_i := \alpha$,
a pour image verticale $t^{d_I}$ avec $d_I = |\alpha| = \sum\limits_{i \in I} d_i$
\end{center}

\medskip

Examinons les composantes de nos modules fétiches :
commençons par le $\bfA[\uX]$-module gradué $E = \Mmac_k$.
C'est en particulier un $\bfA$-module fine-gradué avec comme $\gamma$-composante 
$E_\gamma = \bigoplus \bfA X^\beta e_J$
où $\beta$ et $J$ sont tels que $X^\beta e_J \in \Mmac_k$ et $\beta + D(J) = \gamma$.
Les composantes du $\bfA$-module fine-gradué $E = \Smac_k$ s'obtiennent de la même façon.
Et enfin, pour le $\bfA[\uX]$-module gradué $E = \Jex_h$, 
on a $E_\gamma = \bigoplus \bfA X^\alpha$ où $X^\alpha \in \Jex_h$ et $\alpha = \gamma$ ;
autrement dit 
$$
E_\gamma = 
\left\{
\begin{array}{ll}
\bfA X^\gamma & \text{si $\# \DivSeq(X^\gamma) \geqslant h$} \\ 
0 & \text{sinon}
\end{array}
\right.
$$

\bigskip

Les séries de ces modules s'obtiennent naïvement par les sommes 
$$
\Serie(\Mmac_k) = 
\sum_{\beta,J} X^\beta X^{D(J)}
\qquad 
\Serie(\Smac_{k}) = 
\sum_{\alpha,I} X^\alpha X^{D(I)}
\qquad 
\Serie(\Jex_h) = 
\sum_{\gamma} X^\gamma
$$
où la première somme porte sur les $\beta,J$ tels que $X^\beta e_J \in \Mmac_k$, 
la deuxième sur les $\alpha, I$ tels que $X^\alpha e_I \in \Smac_k$ et la troisième sur les $\gamma$ 
tels que $\#\DivSeq(X^\gamma) \ge h$.

\bigskip

En fait, toutes ces séries obéissent à une petite mécanique.
Prenons $E$ un sous-module monomial de l'algèbre extérieure.
Typiquement, $E$ est égal à $\rmK_k$, ou à $\Mmac_k$, $\Smac_k$, ou encore à $\Jex_h$.
La série d'un tel module, disons $E = \bigoplus_{(\alpha, I) \in \mathscr E} \bfA X^\alpha e_I$ 
où $\mathscr E$ est un sous ensemble de $\bbN^n \times \calP(\{1,\dots,n\})$, 
est $\sum_{(\alpha, I) \in \mathscr E} X^\alpha X^{D(I)}$.

La justification en est la suivante.
Par définition, on a 
$$
\Serie(E) \ = \
\sum_{\gamma \in \bbN^n} \dim_{\bfA}(E_\gamma) \, T^\gamma
\ \in\, \bbZ[[\uT]]
\text{ ou encore }
\Serie(E) \ = \
\sum_{\gamma \in \bbN^n} \dim_{\bfA}(E_\gamma) \, X^\gamma
$$
où $\dim (E_\gamma)$ est le nombre de $(\alpha, I) \in \mathscr E$ tels que 
$\alpha + D(I) = \gamma$.
Ainsi :
$$
\Serie(E) \ = \
\sum_{\gamma \in \bbN^n} 
\sum_{(\alpha, I) \in \mathscr E \atop \alpha + D(I) = \gamma} 
\, X^\gamma
\ = \ 
\sum_{\gamma \in \bbN^n} 
\sum_{(\alpha, I) \in \mathscr E \atop \alpha + D(I) = \gamma} 
\, X^{\alpha + D(I)} 
\ = \ 
\sum_{(\alpha, I) \in \mathscr E}  X^{\alpha}X^{D(I)} 
$$

Résumons, on a l'égalité fondamentale suivante 
$$
\Serie\biggl(\bigoplus_{(\alpha, I) \in \mathscr E} \bfA X^\alpha e_I\biggr)
\ = \ 
\sum_{(\alpha, I) \in \mathscr E} X^\alpha X^{D(I)}
\leqno (\spadesuit)
$$

\begin{rmq}
On voit que le $\bfA$-module $\Smac_k$ est un sous-$\bbN^n$-module gradué de
$\rmK_k$ sans en être un \mbox{sous-$\bfA[\uX]$-module} gradué.
\`A noter que l'on dispose de deux $\bfA[\uX]$-modules gradués $\rmK_k/\Mmac_k$ et
$\Mmac_{k-1}$ qui sont isomorphes comme $\bfA$-modules $\bbN^n$-gradués et
également isomorphes au $\bfA$-module $\bbN^n$-gradué $\Smac_k$.
En revanche, les $\bfA[\uX]$-modules
$\rmK_k/\Mmac_k$ et $\Mmac_{k-1}$ ne sont pas isomorphes; par exemple, la multiplication
par $X_1$ est injective sur $\Mmac_{k-1}$ (car $\Mmac_{k-1}$ est contenu dans le
module libre $\rmK_{k-1}$ et que $X_1$ est régulier) ; mais elle ne l'est pas en
général sur $\rmK_k/ \Mmac_k$ car on peut avoir $X_1(X^\alpha e_I) \in \Mmac_k$
sans que $X^\alpha e_I$ ne soit dans~$\Mmac_k$. Prenons par exemple $X^{\alpha} = 1$,
$d_1=1$ et $I$ ne contenant pas 1: on a alors $X_1(X^\alpha e_I) = X_1e_I \in \Mmac_k$
mais $X^\alpha e_I = e_I \notin \Mmac_k$.
\end{rmq}

\subsection{Série du terme $\rmK_k$ (Koszul), du sous-module de Macaulay $\Mmac_k$, de son supplémentaire $\Smac_k$
et de l'idéal $\Jex_h$}

Les dénominateurs qui vont intervenir dans nos séries
d'Hilbert-Poincaré sont toujours les mêmes ; il s'agit de :
$$
\text{fine-grading} : (1 - X_1)\cdots (1-X_n)  \qquad\qquad
\text{raw-grading} : (1-t)^n
$$
De ce fait, nous fournirons le plus souvent uniquement le \og numérateur\fg,
en le notant $\scrN_f$ en fine-grading et $\scrN_r$ en raw-grading.
Remarquons au passage que $\scrN_f(\bfA[\uX]) = \scrN_r(\bfA[\uX]) = 1$ (confer le 
tout premier paragraphe de la section).

\label{NOTA18-fNum}%
\label{NOTA18-rNum}%
%
%

\begin {rmqs}[A propos de la terminologie ``numérateur'']
\leavevmode

\medskip
$\rhd$
Elle ne désigne pas le numérateur de la fraction rationnelle lorsque
l'on écrit celle-ci sous forme irréductible.  Prenons comme exemple le
module gradué $\Jex_h/\Jex_{h+1}$. Sa série $\bbN$-graduée est la
fraction rationnelle suivante \emph {écrite sous forme irréductible}:
$$
\Serie_r(\Jex_h/\Jex_{h+1}) = \dfrac{N_h(t)}{(1-t)^h}
$$
où $N_h(t)$ est un certain polynôme à coefficients entiers positifs, étranger à $(1-t)^h$.
Le fait qu'il soit étranger se traduit par $N_h(1) \ne 0$. Ici, en désignant par $e_\ell$ la
fonction symétrique élémentaire de degré~$\ell$ en $n$ variables, on dispose
d'un certificat du caractère étranger sous la forme suivante:
$$
N_h(1) = e_{n-h}(d_1, \cdots, d_n)
$$
En ce qui concerne les assertions relatives à $\Jex_h/\Jex_{h+1}$, cf.
dans la section~\ref{SectionBriques} (modules-briques), le dernier exemple.

\medskip
$\rhd$
Pour les modules gradués qui nous concernent, la forme irréductible
$N(t)/(1-t)^m$ de leur série $\bbN$-graduée (où~$N$ est un polynôme à
coefficients entiers vérifiant $N(1) \ge 1$) joue un rôle important en
algèbre commutative, car porteuse d'invariants structurels du
module. Par exemple, l'entier $m$, qui est donc l'ordre du pôle $t=1$
de la fraction rationnelle, est la dimension de Krull du module (égale
à la dimension de Krull de son idéal annulateur).  D'autres invariants
(multiplicité, degré) sont attachés au polynôme $N$ par
l'intermédiaire de son degré, son coefficient dominant, son évaluation
au point~1, etc.

\smallskip

Nous n'étudierons pas cet aspect ici. La détermination de la série
d'Hilbert-Poincaré d'un module gradué sera principalement utilisée à
des fins dimensionnelles (dimension des composantes homogènes),
informations susceptibles de fournir dans certains cas (selon le
module gradué) le poids en~$P_i$ de divers scalaires (souvent des
déterminants) attachés à~$\uP$. Par ailleurs, puisque les modules
qui interviennent sont monomiaux, nous nous sommes intéressés à
la combinatoire sous-jacente à la détermination de leurs séries
d'Hilbert-Poincaré.

\end {rmqs}

\medskip

Commençons par déterminer la série du module $\rmK_k$ issu du complexe de Koszul associé au format $D$ (nous
pensons que, parmi tous les modules que nous allons traiter, c'est le plus simple !).

\begin {prop} [Séries d'Hilbert-Poincaré de $\rmK_k$ (Koszul)]
\label{SerieKk}
Pour $k$ fixé, les séries (fine et raw-grading) de $\rmK_k$ ont pour numérateurs:
$$
\scrN_f(\rmK_k) = \sum_{\#I=k} X^{D(I)},
\qquad\qquad
\scrN_r(\rmK_k) = \sum_{\#I=k} t^{d_I}
$$
\end {prop}

\begin {proof}

Commençons par déterminer la série fine grading ; il suffira de réaliser $X_i = t$ pour 
obtenir la série raw-grading.
D'après l'égalité $(\spadesuit)$ de la page précédente, on a 
$$
\Serie(\rmK_k) \ = \ 
\sum_{\alpha \in \bbN^n \atop \# I = k} X^\alpha X^{D(I)}
$$
Comme les conditions sur $\alpha$ et $I$ sont indépendantes, on a :
$$
\Serie(\rmK_k) \ = \ 
\Big(\sum_{\alpha \in \bbN^n} X^\alpha\Big) 
\Big(\sum_{\# I = k} X^{D(I)}\Big)
$$
La première série vaut $\dfrac{1}{(1-X_1) \cdots (1-X_n)}$.
On a donc :
$$
\Serie(\rmK_k) \ = \ 
\dfrac{\sum_{\#I=k} X^{D(I)}}{(1-X_1) \cdots (1-X_n)}
$$
d'où le numérateur annoncé.

\medskip

Voici une autre justification, pour la série raw-grading, utilisant des 
résultats sur les séries de modules twistés, à savoir :
$$
\Serie\Big(\bigoplus_\ell E_\ell\Big) = \sum_\ell \Serie(E_\ell)
\qquad
\Serie\big(E(-m)\big) = t^m \Serie(E)
$$
Ces deux formules proviennent d'une part du fait que la composante homogène de degré $d$ 
d'une somme directe est la somme directe des composantes homogènes de degré $d$, 
et d'autre part de la \textit{venerable formulae} $E(-m)_d = E_{-m+d}$.

Achevons à présent la détermination de la série raw-grading.
Comme $e_I$ est de degré $d_I = \sum_{i\in I} d_i$: 
$$
\rmK_k = \bigoplus_{\#I = k} \bfA[\uX]\,e_I
\ \simeq \ \bigoplus_{\# I = k} \bfA[\uX](-d_I)
$$
on en déduit 
$$
\Serie(\rmK_k) \ = \ 
\sum_{\# I = k} t^{d_I} \, \Serie(\bfA[\uX])
$$
Comme $\Serie(\bfA[\uX]) = \dfrac{1}{(1-t)^n}$, on en déduit le numérateur annoncé.
\end {proof}

\begin{rmq}
Par définition de la série en $t$ raw graduée, l'entier $\dim\rmK_{k,d}$ 
est le coefficient en~$t^d$ de $\Serie_r(\rmK_k)$.
Or, en~\ref{dminKk}, nous avons fourni une formule binomiale pour $\dim\rmK_{k,d}$
qui conduit donc à l'identité suivante:
$$
\dfrac{\sum_{\#I=k} t^{d_I}} {(1-t)^n} 
\ =\ 
\sum_{\#I=k\atop d \in \bbN} \binom{n+d-d_I-1}{d-d_I}\, t^d
$$
Cette identité est étroitement liée à la suivante (pour $m\in \bbN$):
$$
\dfrac{t^m} {(1-t)^n}  \ =\  \sum_{d \in \bbN} \binom{n+d-m-1}{d-m}\,t^d
\qquad \text{en particulier} \qquad
\dfrac{1}{(1-t)^n} \ = \  \sum_{d \in \bbN} \binom{n+d-1}{d}\, t^d
$$  
Il est important de signaler de nouveau que nous utilisons, pour les
coefficients binomiaux, la convention adoptée depuis la proposition
\ref{dminKk}. Ainsi, dans le développement en série figurant à gauche,
on ne peut pas, au prétexte que $(d-m) + (n-1) = n+d-m-1$, remplacer
l'indice bas du coefficient binomial, égal à $d-m$, par $n-1$.
Quitte à insister lourdement, avec la
convention adoptée, pour $m \ge n$: 
$$
S_{n,m} := 
\sum_{d \in \bbN} \binom{n+d-m-1}{n-1}\,t^d
\qquad \text{n'est pas égale à} \qquad
\dfrac{t^m} {(1-t)^n}
\leqno \text{\dbend}
$$
En effet, pour $0 \le d \le m-n$, le coefficient en $t^d$ de $S_{n,m}$ est un coefficient binomial
\emph{non nul} (malgré le fait que son indice haut soit $<0$!).

\end{rmq}

\subsubsection*{Séries du sous-module de Macaulay $\Mmac_k$ et de son supplémentaire
$\Smac_k$ ($k \ge 0$)}

Les deux familles de sous-modules $\bbN^n$-gradués $\Mmac_\sbullet$ et
$\Smac_\sbullet$ ne sont pas indépendantes.  D'une part, $\Mmac_k
\simeq \Smac_{k+1}$ de manière graduée, donc les deux modules ont même
série d'Hilbert-Poincaré; ainsi la détermination des séries des
$\Smac_\sbullet$ fournit celles des $\Mmac_\sbullet$.
D'autre part, $\rmK_k = \Mmac_k \oplus \Smac_k$ et comme nous avons déterminé
la série d'Hilbert-Poincaré de $\rmK_k$ (cf. la proposition~\ref{SerieKk}),
la détermination de la série de l'un des deux parmi
$\Mmac_k$ ou $\Smac_k$ fournit l'autre. 

\medskip

Nous avons choisi ici de traiter d'abord les modules $\Smac_\sbullet$,
modules porteurs des rangs attendus.  Commençons par montrer en quoi
l'interaction entre les deux familles évoquée ci-dessus permet la
détermination des séries en nous appuyant sur~$\Smac_k$.  Notons
provisoirement $s_k = \Serie(\Smac_k)$.  Puisque nous disposons d'un
isomorphisme gradué:
$$
\rmK_k \ =\ 
\Mmac_k \oplus \Smac_k 
\ \simeq \ 
\Smac_{k+1} \oplus \Smac_k
$$
nous avons $\Serie(\rmK_k) = s_{k+1} + s_k$ puis
$$
s_k = \Serie(\rmK_k) - s_{k+1} = \Serie(\rmK_k) - \Serie(\rmK_{k+1}) + s_{k+2}
= \cdots
$$
Comme $s_i = 0$ pour $i > n$, 
nous venons d'exprimer à peu de frais la série de $\Smac_k$ comme une somme alternée
de séries de termes Koszul: \label{FormuleAlterneeSerieSk}
$$
\Serie(\Smac_k) \ =\ 
\sum_{\ell=k}^n (-1)^{\ell-k}\, \Serie(\rmK_{\ell})
$$
En considérant la raw-graduation et le coefficient en $t^d$ de ces séries,
on retrouve l'égalité pour le rang attendu $r_{k,d}$:
$$
r_{k,d} = \dim \Smac_{k,d} = \sum_{\ell=k}^n (-1)^{\ell-k}\, \dim(\rmK_{\ell,d})
$$
De l'égalité sur les séries, pour les numérateurs, en fine-grading ou raw-grading
$$
\scrN(\Smac_k) \ = \ 
\sum_{\ell=k}^n (-1)^{\ell-k}\, \scrN(\rmK_\ell)
$$
Comme nous connaissons $\scrN(\rmK_\ell)$, nous pourrions en rester là.
Mais nous allons fournir d'autres formules plus économiques.

On rappelle, pour une partie $J \subset \{1...n\}$, les notations $X^{D(J)} =
\prod\limits_{j \in J} X_j^{d_j}$ et $\overline{X}^{D(J)} =\prod\limits_{j \in J} (1-X_j^{d_j})$.

\begin{prop}[Série de $\Smac_k$] 
\label{SerieSk}
Pour $k \ge 0$, la série de $\Smac_k$ a pour numérateurs :
$$
\scrN_f(\Smac_{k})  
\ = \ 
\sum_{\#I \, = \, k } \ 
\overline{X}^{D([1...\min I[)} X^{D(I)} 
\qquad \text{et }
\qquad 
\scrN_r(\Smac_{k})  
\ = \ 
\sum_{\# I \, = \, k} \ \ 
t^{d_I} \!\!
\prod_{j < \min I} (1- t^{d_j})
$$
\end{prop}

\medskip

Avant la preuve, une remarque concernant le cas particulier $k= 0$
i.e. $\Smac_0 = \bigoplus_{\alpha \preccurlyeq \emouton} \bfA
X^\alpha$.  Pour cette valeur de $k$, le numérateur fait intervenir la
seule partie de cardinal $0$, à savoir $I = \emptyset$.  Notre convention $\min
\emptyset = n+1$ conduit à $\overline{X}^{D([1...\min I[)} X^{D(I)} =
\prod_{j=1}^n(1-X_j^{d_j})$, si bien que la série de $\Smac_0$ est
le \emph{polynôme}
$$
\Serie(\Smac_0) = \prod_{j=1}^n (1 + X_j + X_j^2 + \cdots + X_j^{d_j-1})
$$

\begin{proof}\leavevmode

D'après l'égalité $(\spadesuit)$, la série de $\Smac_k$ est  
$\sum_{\alpha, I} X^\alpha X^{D(I)}$
où la somme porte sur les $\alpha, I$ tels que 
$X^\alpha e_I \in \Smac_k$, c'est-à-dire avec $\#I = k$ et $\minDiv(X^\alpha) \ge \min I$. 
Cette somme ne se simplifiant pas \textit{a priori}, 
partitionnons ces monômes extérieurs \og suivant $I$ \fg{} ; 
autrement dit, 
notons $\Smac_k[I]$ le $\bfA$-module de base les $X^\alpha e_I \in \Smac_k$ de sorte que :
$$
\Smac_k \ = \ 
\bigoplus_{\#I = k} \Smac_k[I]
\quad \text{d'où } \quad 
\Serie(\Smac_k) \ = \ 
\sum_{\#I = k} \Serie(\Smac_k[I])
$$
Un monôme extérieur est dans $\Smac_k[I]$ s'il s'écrit $X^\alpha e_I$ avec 
$\alpha$ vérifiant $\minDiv(X^\alpha) \ge \min I$, 
c'est-à-dire vérifiant $\alpha_j < d_j$ pour tout $j < \min I$.
En notant $B_I \subset \bbN^n$ l'ensemble de ces $\alpha$, on a donc 
$$
\Smac_k[I] = \bigoplus_{\alpha \in B_I} \bfA X^\alpha e_I
\qquad \text{ puis } \qquad
\Serie\big(\Smac_k[I]\big) \ = \ 
X^{D(I)} \sum_{\alpha \in B_I} X^\alpha 
$$
\`A la page~\pageref{ChapSeries}, on a presque déterminé cette somme portant sur $B_I$.
Reprenons les mêmes idées :
considérons le produit suivant 
(les $j$\up{ers} facteurs avec $j < \min I$ sont des \textit{polynômes}, les autres facteurs 
sont de \og vraies séries \fg{})
$$
\prod_{j < \min I}
(1+X_j + \cdots + X_j^{\alpha_j} + \cdots + X_j^{d_j-1})
\quad \times \quad 
\prod_{j \ge \min I}
(1+X_j + \cdots + X_j^{\alpha_j} + \cdots)
$$
produit qui, après développement, 
est égal à la somme des monômes $X^\alpha$ 
vérifiant la condition $\alpha_j < d_j$ pour tout $j < \min I$. 
Autrement dit, ce produit est égal à la somme $\sum_{\alpha \in B_I} X^\alpha$.
D'où 
$$
\sum_{\alpha \in B_I} X^\alpha \ = \ 
\dfrac{\prod\limits_{j < \min I} (1-X_j^{d_j})}{(1-X_1) \cdots (1-X_n)}
\quad \text{qui s'écrit encore } \quad
\dfrac{\overline{X}^{D([1..\min I[)}}{(1-X_1) \cdots (1-X_n)} 
$$
On a donc déterminé le numérateur de $\Serie(\Smac_{k}[I])$, il s'agit de  :
$$
\scrN_f(\Smac_{k}[I])
\ = \ 
X^{D(I)} \overline{X}^{D([1..\min I[)}
$$
Il ne reste plus qu'à sommer sur $I$, pour obtenir :
$$
\scrN_f(\Smac_{k})
\ = \ 
\sum_{\#I \, = \, k } 
X^{D(I)} 
\scrN_f(\Smac_{k}[I])
\ = \ 
\sum_{\#I \, = \, k } X^{D(I)} \overline{X}^{D([1...\min I[)}
$$
ce qui est bien la formule annoncée.
\end{proof}

\begin{prop}[Série de $\Mmac_k$] 
\label{SerieMk}

Pour $k \ge 0$, la série de $\Mmac_k$ a pour numérateur:
$$
\scrN_f(\Mmac_{k})  
\ = \ 
\sum_{\#I \, = \, k+1} \ 
\overline{X}^{D([1...\min I[)} X^{D(I)} 
\ = \ 
\sum_{i,J}
\overline{X}^{D([1...i[)} X_i^{d_i}X^{D(J)}  
$$
où la double somme porte sur $i \in \{1,\dots, n\}$ et $J \subset
\{1,\dots, n\}$ de cardinal $k$, avec la contrainte $J \subset
\{i+1,\dots, n\}$ (ou encore avec la contrainte $i < \min J$).

\medskip
La double somme peut encore être écrite sous la forme suivante:
$$
\scrN_f(\Mmac_{k})  
\ = \ 
\sum_{\#J=k} X^{D(J)} \sum_{i<\min J} X_i^{d_i} \overline{X}^{D([1...i[)}
\ = \ 
\sum_i X_i^{d_i} \overline{X}^{D([1...i[)} 
\sum_{J\subset\{i+1..n\}\atop\#J=k} X^{D(J)}  
$$
\end{prop}

\begin {proof} \leavevmode

La première égalité résulte de l'existence d'un isomorphisme
$\bbN^n$-gradué $\Mmac_k \simeq \Smac_{k+1}$.
La deuxième formule avec la double somme
s'obtient en posant $i = \min I$ pour une partie $I$ de cardinal $k+1$,  
et en écrivant $I$ sous la forme $\{i\} \vee J$.
\end{proof}

\medskip

Pour en finir avec les $\Smac_\sbullet$ et $\Mmac_\sbullet$, une remarque
a priori surprenante. Nous venons d'établir, pour $k \ge 1$ :
$$
\scrN_r(\Mmac_{k-1}) = \sum_{\#I=k} t^{d_I} \prod_{j < \min I} (1-t^{d_j})
$$
Mais nous affirmons que nous avons également (attention à $k$ versus $k-1$):
$$
\scrN_r(\Mmac_{k}) = \sum_{\#I=k} t^{d_I} \Bigr(1- \prod_{j < \min I}(1-t^{d_j})\Bigr)
$$
En effet, la somme des deux membres droits est:
$$
\sum_{\#I=k} t^{d_I} \text{ qui n'est autre que }  \scrN_r({\rmK_k})
$$
Comme $\Mmac_{k-1} \oplus \Mmac_k \simeq \rmK_k$, on a 
$\scrN_r(\Mmac_{k}) = \scrN_r({\rmK_k}) - \scrN_r(\Mmac_{k-1})$, 
ce qui permet de conclure.

\subsubsection*{Série de l'idéal $\Jex_h$}

\begin{prop}[Série de $\Jex_h$]
\label{SerieJh}
Pour $h \ge 0$, la série de $\Jex_h$ a pour numérateur :
$$
\scrN_f(\Jex_h)  
\ = \ 
\sum_{\#I\,\geqslant\,h} \ 
X^{D(I)}  \overline{X}^{D(\overline I)}
\qquad 
\text{ et } 
\qquad 
\scrN_r(\Jex_h)  
\ = \ 
\sum_{\#I\,\geqslant\,h} \
t^{d_I} \prod_{j \notin I} (1- t^{d_j})
$$
On dispose aussi de la formule \og économique \fg{} suivante :
$$
\scrN_f(\Jex_h)  
\ = \ 
\displaystyle 
\sum_{\# J \, = \, h} \ 
X^{D(J)} \  \overline{X}^{D([1..\max J[ \,\cap\, \overline J)}
\qquad 
\text{ et } 
\qquad 
\scrN_r(\Jex_h)  
\ = \ 
\sum_{\# J \, = \, h} \ 
t^{d_J} \!\!\!\! \prod_{j \notin J \atop j < \max J} (1- t^{d_j})
$$
La qualification \og économique \fg{} est due au fait que l'ensemble d'indices
intervenant dans la somme est plus petit que celui de la somme précédente.

\medskip

Pour $h=1$, voici la formule économique, suivie d'une \emph{autre} formule:
$$
\scrN_f(\Jex_1) \ =\ 
\sum_{i=1}^n X_i^{d_i} \prod_{j<i} (1-X_j^{d_j}) 
\ =\ 
1 - \prod_{i=1}^n (1-X_i^{d_i}) 
$$
Laquelle est vraiment la plus économique?
\end{prop}

\begin{proof} \leavevmode

$\rhd$
En notant $\Jex_h[I]$ le $\bfA$-module de base les
$X^\alpha \in \Jex_h$ vérifiant $\DivSeq(X^\alpha) = I$, on a :
$$
\Jex_h = \bigoplus_{\# I \ge h} \Jex_h[I]
\qquad \text{d'où} \qquad
\Serie(\Jex_h)
\ = \ 
\sum_{\#I \ge h} 
\Serie(\Jex_h[I])
$$
Un monôme est dans $\Jex_h[I]$ si, et seulement si, il peut s'écrire $X^{D(I)}X^\beta$ 
avec, pour tout $j \notin I$, $\beta_j < d_j$. En notant $B_I \subset \bbN^n$ l'ensemble
de ces $\beta$:
$$
\Jex_h[I] = X^{D(I)}\, E_I
\qquad \text{avec} \qquad
E_I = \bigoplus_{\beta \in B_I} \bfA\,X^\beta
$$
Pour déterminer la somme $s_I := \Serie(E_I)$, 
il suffit d'utiliser la même technique que dans la preuve précédente.
Considérons le produit suivant 
(le premier est un produit de \textit{polynômes}) :
$$
\prod_{j \in \overline{I}}
(1+X_j + \cdots + X_j^{\beta_j} + \cdots + X_j^{d_j-1})
\quad \times \quad 
\prod_{j \in I}
(1+X_j + \cdots + X_j^{\alpha_j} + \cdots)
\leqno (\star)
$$
Son développement est égal à la somme des monômes
$X^\beta$ vérifiant la condition $\beta_j < d_j$ pour tout
$j \in \overline I$, c.a.d à la somme convoitée~$s_I$.
Or ce produit $(\star)$ de \textit{séries} est la \textit{fraction rationnelle}
$$
\dfrac{N}{(1-X_1) \cdots (1-X_n)}
\qquad \text{avec} \qquad
N \ = \ \prod\limits_{j \in \overline I} (1-X_j^{d_j}) \ \overset {\rm def}{=}\ 
\overline{X}^{D(\overline{I})}
$$
Comme $N$ est le numérateur de $\Serie(E_I)$, on a :
$$
\scrN_f(E_I) = \overline{X}^{D(\overline{I})}
\qquad \text{puis} \qquad
\scrN_f\big(\Jex_h[I]\big) = X^{D(I)} \scrN_f(E_I) =
X^{D(I)}\, \overline{X}^{D(\overline{I})}
$$
En utilisant que $\Jex_h$ est la somme directe des $\Jex_h[I]$ (en début de preuve),
on obtient la formule annoncée:
$$
\scrN_f(\Jex_h) \ = \  \sum_{\#I\ge h} \scrN_f\big(\Jex_h[I]\big) =
\sum_{\#I\ge h} X^{D(I)}\, \overline{X}^{D(\overline{I})}
$$
Une façon plus directe et légère de raconter ce qui vient de se passer
est de dire que nous avons paramétré les monômes $X^\alpha$ de
$\Jex_h$ par la suite $I = \DivSeq(X^\alpha)$, ce qui a conduit à
la formule ci-dessus.


\bigskip
$\rhd$
Passons à la formule qualifiée \og d'économique \fg{}.  Cette fois, on
paramètre les $X^\alpha \in \Jex_h$ par l'ensemble $J
= \DivSeq(X^\alpha)[1..h]$, de cardinal $h$, de sorte que les $h$
premiers termes de $\DivSeq(X^\alpha)$ sont ceux de $J$.  Ainsi,
$X^\alpha = X^{D(J)}X^\beta$ avec $\beta_j < d_j$ pour tout $j \in
[1..\max J[ \,\cap\,\overline J$. Pour un $J$ fixé de cardinal $h$,
notons $B'_J \subset \bbN^n$ l'ensemble de ces $\beta$ et introduisons
$\Jex'_h[J]$ le $\bfA$-sous-module de $\Jex_h$ de base les
$X^\alpha \in \Jex_h$ vérifiant $\DivSeq(X^\alpha)[1..h] = J$.  Alors,
d'une part:
$$
\Jex_h = \bigoplus\limits_{\#J = h} \Jex'_h[J]
\qquad \text{d'où} \qquad
\Serie(\Jex_h)
\ = \
\sum_{\#J = h}
\Serie(\Jex'_h[J])
$$
Et d'autre part:
$$
\Jex'_h[J] = X^{D(J)}\, E'_J
\qquad \text{avec} \qquad
E'_J = \bigoplus_{\beta \in B'_J} \bfA\,X^\beta
$$
Nous ne détaillons pas la détermination de la somme $s'_J
:= \Serie(E'_J)$, qui est analogue à celle de la somme $s_I$ du point
précédent. Nous obtenons pour son numérateur:
$$
\scrN_f(E'_J) = \prod_{j \in [1..\max J[\,\cap\,\overline J} (1-X_j^{d_j})
\quad \overset{\rm def}{=} \quad
\overline{X}^{D([1..\max J[\,\cap\,\overline J)}
$$
En utilisant $\scrN_f\big(\Jex'_h[J]\big) = X^{D(J)}\scrN_f(E'_J)$ et
le fait que $\Jex_h$ est la somme directe des
$\big(\Jex'_h[J]\big)_{\#J=h}$, nous obtenons la formule convoitée:
$$
\scrN_1(\Jex_h) \ = \ 
\displaystyle 
\sum_{\# J \, = \, h} \ 
X^{D(J)} \  \overline{X}^{D([1..\max J[ \,\cap\, \overline J)}
$$

\bigskip
$\rhd$
Enfin, en ce qui concerne $\Jex_1$, la première égalité est une conséquence directe de la
formule économique. Et la seconde résulte de $\bfA[\uX] = \Jex_1 \oplus \Smac_0$
donc $\scrN_f(\Jex_1) = 1 - \scrN_f(\Smac_0)$.
\end{proof}

\begin{rmqs}
\leavevmode

$\rhd$
Pour $h=0$, puisque $\Jex_0 = \bfA[\uX]$,
on doit avoir $\sum_I X^{D(I)} \overline{X}^{D(\overline I)} = 1$
où la somme porte sur toutes les parties $I$ de $\{1...n\}$.
C'est bien le cas, puisque c'est une somme du type
$\sum_I \prod_{i \in I} a_i \prod_{j \notin I} (1-a_j)$,
qui vaut $\prod_{i =1}^n \big(a_i + (1-a_i)\big)$, c'est-à-dire $1$.

\medskip

$\rhd$
On a justifié l'égalité des 2 formules fournissant~$\scrN_f(\Jex_1)$ via
l'argument structurel $\bfA[\uX] = \Jex_1 \oplus \Smac_0$.
Ne pas oublier pour autant l'efficacité d'un calcul.
Ainsi, le lecteur pourra remarquer qu'en notant $u_i = 1-X_i^{d_i}$, cette
égalité est du type:
$$
\sum_{i=1}^n v_i = 1 - \prod_{i=1}^n u_i
\qquad\text{où}\qquad
v_i = (1-u_i) \prod_{j<i}u_j \ = \ \prod_{j<i}u_j - \prod_{j<i+1}u_j  
$$
A droite, l'expression de $v_i$ fait apparaître une somme télescopique.
Quitte à insister lourdement:
$$
\begin {array}{ccrll}
v_1  &=& 1 &-& u_1 \\
v_2  &=& u_1 &-& u_1u_2 \\
v_3  &=& u_1u_2 &-& u_1u_2u_3 \\
     &\vdots& \\
v_n  &=& u_1\cdots u_{n-1} &-& u_1 u_2\cdots u_n \\
\end {array}
$$

\medskip

$\rhd$
Dans le même genre, nous laissons le soin au lecteur de prouver
simplement:
$$
\sum_{\#I\,\geqslant\,h}\ X^{D(I)}\overline{X}^{D(\overline I)} =
1 - \sum_{\#J\,\leqslant\,h-1}\ X^{D(J)}\overline{X}^{D(\overline J)}
$$
\end{rmqs}

\subsection{Bonus : formules économiques}
\label{BonusFormulesEco}

On a constaté que l'on obtenait facilement les égalités
(cf. page~\pageref{FormuleAlterneeSerieSk} et la
proposition~\ref{SerieKk}) :
$$
\scrN(\Smac_k) = \sum_{\ell\ge k} (-1)^{\ell-k}\, \scrN(\rmK_\ell),
\qquad \text{et } \qquad
\scrN(\rmK_\ell) = \sum_{\#I=\ell} X^{D(I)} 
$$
La question est de savoir si cela ne nous permettrait pas de retrouver
la formule annoncée en~\ref{SerieSk}, que l'on peut qualifiée
d'économique en comparaison de celle utilisant la somme alternée.  La
réponse va être positive. Mais avant, il nous faut dégager des notions
générales.

\medskip

\'Etant donné un ensemble fini  $F$, un anneau commutatif $\bfR$ et
une fonction $f : F \to \bfR$, on définit deux \og extensions multiplicatives\fg{} à l'ensemble
des parties de $F$, l'une utilisant $f$, l'autre  $1-f$.
Pour $I \subset F$, on pose : 
$$
\pi_0(I) = \prod_{i\in I} f(i), \qquad \pi_1(I) = \prod_{i\in I} \big(1-f(i)\big) 
$$
On a bien sûr $\pi_0(\emptyset) = \pi_1(\emptyset) = 1$, $\pi_0(\{x\}) = f(x)$ et
$\pi_1(\{x\}) = 1-f(x)$.

\label{NOTA18-pi0}%
\label{NOTA18-pi1}%

\medskip

Prenons par exemple $F = \{1..n\}$ et pour $f$ la fonction $i \mapsto X_i^{d_i}$. Alors 
on retrouve des expressions déjà croisées, puisque :
$$
\pi_0(I) = X^{D(I)}, \qquad  \pi_1(I) = \overline X^{D(I)}
$$

\medskip

Lorsque $F$ est totalement ordonné, on note $F_{<\min I}$ la partie de
$F$ constituée des éléments~$<\min I$.  Puisque $I_1 \subseteq
I_2 \Rightarrow \min(I_1) \ge \min(I_2)$, on convient, pour $I
= \emptyset$, de $F_{<\min I}= F$ pour le motif suivant:
$\min(\emptyset)$ est un élément fictif strictement plus grand que
tout élément de $F$.  De la même manière, on désigne par $F_{<\max I}$
le sous-ensemble de $F$ constitué des éléments $<\max I$.  Comme
$I_1 \subseteq I_2 \Rightarrow \max(I_1) \le \max(I_2)$, on convient,
pour $I =
\emptyset$, de $F_{<\max I}= \emptyset$ avec l'alibi suivant:
$\max(\emptyset)$ est un élément fictif strictement plus
petit que tout élément de $F$.

Mais la vraie raison de cette convention est de rendre valide dans
tous les cas de figure les égalités des lemmes suivants.

\medskip

Nous allons établir dans ces lemmes un certain nombre d'identités
reliant sommes et produits de $\pi_0$ et $\pi_1$. On peut considérer
que toutes ces identités reposent en partie sur le \emph{truc}
suivant:
$$
\sum_{I\subseteq F} \prod_{i\in I} a_i \prod_{j\notin I} b_j =
\prod_{\ell\in F} \big(a_\ell+b_\ell\big)
\leqno (\heartsuit)
$$
Pour un entier $k$, prouvons par exemple l'identité (dans laquelle
$I,J \subset F$ et $\overline I = F\setminus I$):
$$
\sum_{\#I \ge k} \pi_0(I)\pi_1(\overline I) \ =\ 
1 -\!\!\sum_{\#J \le k-1} \pi_0(J)\pi_1(\overline J)
$$
Elle est équivalente à:
$$
\sum_{I \subseteq F} \pi_0(I)\pi_1(\overline I) = 1
$$
Et cette dernière s'obtient à partir de $(\heartsuit)$ en prenant $a_\ell = f(\ell)$ et $b_\ell = 1-f(\ell)$.

\begin {lem} [Sommes alternées de $\pi_0$]
\label{FirstEcoLemma}  
\leavevmode

\begin {enumerate} [\rm i)]
\item
On dispose de l'identité de base:
$$
\sum_{I \subseteq F} (-1)^{\#I} \pi_0(I) = \pi_1(F)
$$
\item
On suppose $F$ totalement ordonné.  Soit $J \subset F$ fixé de
cardinal $k$.  On note $\calR(J)$ l'ensemble des $I\subset F$ de
cardinal $\ge k$ tels que $J$ soit l'ensemble des $k$ \emph {derniers}
éléments de $I$. En conséquence $I \supseteq J$ et
$$
\calR(J) = \big\{ I \supset J \mid  \{\text {$k$ derniers éléments de $I$}\} = J \big\}
$$
Alors
$$
\sum_{I \in \calR(J)} (-1)^{\#I - k}\, \pi_0(I\setminus J) = \pi_1(F_{<\min J})
$$
\item
Pour tout $k \in \bbN$, on a l'identité (dans laquelle $I,J \subseteq F$):
$$
\sum_{\#I\ge k} (-1)^{\#I - k}\,\pi_0(I) 
\ = \ 
\sum_{\#J=k} \pi_0(J)\ \pi_1(F_{<\min J})
$$
\end {enumerate}
\end {lem}

\begin {proof} \leavevmode

i) Dans $(\heartsuit)$, pour $\ell \in F$, prendre $a_\ell = -f(\ell)$ et $b_\ell = 1$.

\medskip
ii)
Puisque chaque $I \in \calR(J)$ contient $J$, on va réaliser le \og
changement de variables\fg{} $K = I\setminus J
\iff I = K \sqcup J$. Que doit vérifier $K$ en plus d'être contenu dans $F \setminus J$?
Pour que $J$ soit égal à l'ensemble des $k$ derniers éléments de $I = K\sqcup J$, il faut et il suffit
que $K \subset F_{<\min J}$. Et comme $F_{<\min J} \subset F\setminus J$, on voit que $K$ varie
dans $F_{<\min J}$. En utilisant $\#I - k = \#I - \#J = \#K$, la somme $s$ de gauche peut donc être réécrite:
$$
s = \sum_{K\subseteq F_{<\min J}} (-1)^{\#K} \pi_0(K)
$$
Et d'après l'identité de base, $s=\pi_1(F_{<\min J})$.

\medskip
iii)
On partitionne les $I$ de cardinal $\ge k$ par l'ensemble $J$
de leurs $k$ derniers éléments. La somme de gauche s'écrit alors:
$$
s = \sum_{\#J = k} s_J
\qquad \text{avec} \qquad
s_J = \sum_{I\in \calR(J)} (-1)^{\#I -k} \,\pi_0(I)
$$
En écrivant $I = J \sqcup (I\setminus J)$ donc $\pi_0(I) = \pi_0(J) \times \pi_0(I\setminus J)$:
$$
s_J = \pi_0(J) \times \sum_{I\in \calR(J)} (-1)^{\#I - k}\, \pi_0(I \setminus J)
$$
D'après le point ii):
$$
s_J = \pi_0(J)\,\pi_1(F_{<\min J})
$$
ce qui termine la preuve car $s$ est la somme des $(s_J)_{\#J=k}$.
\end {proof}

\subsubsection*{Application à la détermination de $\Serie(\Smac_k)$}

Considérons la fonction $f : i \mapsto X_i^{d_i}$ pour laquelle
$$
\pi_0(I) = X^{D(I)}, \qquad\qquad  \pi_1(\{1...n\}_{<\min I}) = \overline X^{D([1...\min I[)}
$$
En début de section, nous avons rappelé que : 
$$
\scrN(\Smac_k) = \sum_{\ell\ge k} (-1)^{\ell-k}\, \scrN(\rmK_\ell),
\qquad \text{et} \qquad
\scrN(\rmK_\ell) = \sum_{\#I=\ell} X^{D(I)} 
$$
Ainsi, $\scrN(\rmK_\ell)$ s'écrit $\displaystyle \sum_{\#I=\ell} \pi_0(I)$, d'où 
$$
\scrN(\Smac_k) \ = \ 
\sum_{\#I \ge k} (-1)^{\#I-k}\,\pi_0(I)
$$
En utilisant le lemme, il vient donc
$$
\scrN(\Smac_k) \ =\ \sum_{\#J= k} \pi_0(J)\pi_1(\{1...n\}_{<\min J})
\ \overset{\rm def}{=}\ 
\sum_{\#J \, = \, k } \ 
X^{D(J)} \overline{X}^{D([1...\min J[)} 
$$
et on retrouve ainsi la formule annoncée en~\ref{SerieSk}.

\bigskip

L'origine du lemme qui suit réside dans l'analyse de la formule
dite économique de la série de~$\Jex_h$, plus précisément de sa preuve (cf.~la proposition~\ref{SerieJh}).

\begin {lem} \label {SecondEcoLemma}
\leavevmode
  
Pour $J \subset F$, nous utilisons, pour son complémentaire dans $F$,
les deux notations $\overline J$ et $F\setminus J$.
Vu les changements à venir de l'ambiant $F$, la seconde notation
est indispensable.

\begin {enumerate}[\rm i)]
\item
Pour une partie fixée $I_0 \subset F$:
$$
\sum_{I\subseteq I_0} \pi_0(I)\pi_1(F\setminus I) = \pi_1(F\setminus I_0),
\qquad \text {en particulier pour $I_0=F$}, \qquad
\sum_{I\subseteq F} \pi_0(I)\pi_1(F\setminus I) = 1
$$
\item
On suppose $F$ totalement ordonné. Soit $J \subset F$ fixé de cardinal
$k$.  On note $\calL(J)$ l'ensemble des $I\subset F$ de cardinal $\ge
k$ tels que $J$ soit l'ensemble des $k$ \emph{premiers} éléments de
$I$. En conséquence $I \supseteq J$ et
$$
\calL(J) = \big\{ I \supset J \mid  \{\text {$k$ premiers éléments de $I$}\} = J \big\}
$$
Alors:
$$
\sum_{I\in \calL(J)} \pi_0(I\setminus J)\pi_1(\overline I) = \pi_1(\overline J \cap F_{<\max J})
$$

\item
Pour tout $k \in \bbN$: 
$$
\sum_{\#I \ge k} \pi_0(I)\pi_1(\overline I) =
\sum_{\#J= k} \pi_0(J)\pi_1(\overline J \cap F_{<\max J})
$$
Pour $k=0$, le membre droit vaut 1 et on retrouve le fait que la somme de gauche est égale à~1.
\end {enumerate}
\end {lem}

\begin {proof} \leavevmode

i)
Dans le \emph {truc} $(\heartsuit)$, en prenant $a_\ell = f(\ell)$ et
$b_\ell = 1-f(\ell)$, on obtient le cas particulier.
Pour le cas général, écrivons:
$$
F\setminus I = (F\setminus I_0) \sqcup (I_0\setminus I)
\qquad \text {d'où} \qquad
\pi_1(F\setminus I) = \pi_1 (F\setminus I_0)\pi_1(I_0\setminus I)
$$
Sa somme $s$ se déduit alors du cas particulier appliqué à $F = I_0$:
$$
s =  \pi_1 (F\setminus I_0) \times \sum_{I\subseteq I_0} \pi_0(I)\pi_1(I_0\setminus I) =
\pi_1(F\setminus I_0) \times 1
$$

\medskip
ii) 
Nous allons voir que \og la solution se passe dans $F' := F\setminus J$\fg. 
\'Ecrivons toute partie $I \supseteq J$ sous la
forme $I = J\sqcup K$ avec $K \subseteq F'$.
Alors $I\in \calL(J)$ si et seulement si $K \subseteq F_{>\max J}$, inclusion
équivalente au fait que $J$ soit l'ensemble des $k$ premiers éléments de $I$.

Réalisons le \og changement de variables\fg{} $K = I\setminus J \subseteq F'$.
En remarquant que $F\setminus I = F'\setminus K$, 
la somme~$s$ peut être indexée par \emph{certains} $K \subseteq F'$, à savoir
ceux vérifiant $K \subseteq F_{>\max J}$. De manière précise,
en introduisant $I_0 \subseteq F'$ défini par
$$
I_0 = F' \cap F_{>\max J},
$$
la somme $s$ peut être réécrite
$$
s = \sum_{K \subseteq I_0} \pi_0(K) \pi_1(F'\setminus K)
$$
D'après le premier point:
$$
s = \pi_1(F'\setminus I_0)
$$
Il n'y a plus qu'à vérifier l'égalité $F'\setminus I_0 = \overline
J \cap F_{<\max J}$.  Il suffit pour cela de revenir aux définitions
avec les notations ad-hoc: $F'=\overline J$ et $I_0 = \overline J \cap
F_{>\max J}$.

\medskip
iii) On partitionne les $I$ de cardinal $\ge k$ par l'ensemble $J$
de leurs $k$ premiers éléments. La somme de gauche s'écrit alors:
$$
s = \sum_{\#J = k} s_J
\qquad \text{avec} \qquad
s_J = \sum_{I\in \calL(J)} \pi_0(I)\pi_1(\overline I) 
$$
En écrivant $I = J \sqcup (I\setminus J)$, on a donc $\pi_0(I) = \pi_0(J) \times \pi_0(I\setminus J)$, puis :
$$
s_J = \pi_0(J) \times \sum_{I\in \calL(J)} \pi_0(I \setminus J)\pi_1(\overline I)
$$
D'après le point ii):
$$
s_J = \pi_0(J)\, \pi_1(\overline J \cap F_{<\max J})
$$
ce qui termine la preuve car $s$ est la somme des $(s_J)_{\#J=k}$.
\end {proof}

\subsubsection*{Application à la détermination d'une formule économique pour $\Serie(\Jex_h)$}

Puisque les identités du lemme tirent leur origine de l'analyse de cette formule économique,
il est bien moral que nous la retrouvions.
En~\ref{SerieJh}, on a obtenu la formule suivante du numérateur de la
série (fine-grading) d'Hilbert-Poincaré de l'idéal~$\Jex_h$:
$$
\scrN_f(\Jex_h) = \sum_{\#I \ge h} X^{D(I)} \overline X^{D(\overline I)}
$$
En prenant comme fonction $f : i \mapsto X_i^{d_i}$, l'expression n'est autre que
le membre gauche du lemme précédent, d'où :
$$
\scrN_f(\Jex_h) =
\sum_{\#I \ge h} \pi_0(I)\,\pi_1(\overline I) \ =\ 
\sum_{\#J= h} \pi_0(J)\,\pi_1(\overline J \cap F_{< \max J})
$$
Ce qui est exactement la formule suivante, déjà annoncée en~\ref{SerieJh} :
$$
\scrN_f(\Jex_h) = \sum_{\#J=h} X^{D(J)}\, \overline X^{D(\overline J \cap [1..\max J[)}
$$

\subsection {Avec du recul: des modules briques pour nos séries d'Hilbert-Poincaré}
\label {SectionBriques}

Pour $I \subset \{1..n\}$, on introduit l'idéal monomial $\Jex_I$ et
son supplémentaire monomial $\brick_I$:
$$
\bfA[\uX] = \Jex_I \oplus \brick_I \qquad
\Jex_I = \langle X_i^{d_i},\ i \in I\rangle, \qquad
\brick_I = \bigoplus_{X^\beta \notin \Jex_I} \bfA X^\beta
$$
Nous avons opté pour la lettre $\brick$ car ces $\brick_I$ nous
servent de briques pour nos séries d'Hilbert-Poincaré. Nous aurions pu
également considérer le quotient qui lui est isomorphe
dans la catégorie des modules $\bbN^n$-gradués:
$$
\bfA[\uX] / \Jex_I  \  \simeq\ \brick_I
$$
A noter que pour $I = \{1...n\}$, l'idéal $\Jex_I$ n'est autre que $\Jex_1$ et que
$\brick_I$, c'est $\Smac_0$. Lorsque $I = \emptyset$, on a $\Jex_\emptyset = \{0\}$
et $\brick_\emptyset = \bfA[\uX]$.

\label{NOTA18-JexI}%
\label{NOTA18-brickI}%
%

\begin {lem}

L'appartenance $X^\beta \in \brick_I$ est caractérisée par $\beta_i < d_i$ pour tout $i \in I$,
ce qui se répercute sur le numérateur de la série d'Hilbert-Poincaré de $\brick_I$:
$$
\scrN_f(\brick_I) = \prod_{i\in I} (1 - X_i^{d_i}) \overset{\rm def}{=}
\overline X^{D(I)}
$$
\end {lem}

\begin {proof}

C'est clair car l'appartenance en question se traduit par
$X^\beta \notin \Jex_I$.  Pour déterminer la série de~$\brick_I$, on
introduit le produit suivant (le premier est un produit
de \textit{polynômes}) :
$$
\prod_{i\in I}
(1+X_i + \cdots + X_i^{\beta_i} + \cdots + X_i^{d_i-1})
\quad \times \quad 
\prod_{j \notin I}
(1+X_j + \cdots + X_j^{\alpha_j} + \cdots)
$$
Son développement est égal à la somme des monômes
$X^\beta$ vérifiant la condition $\beta_i < d_i$ pour tout
$i \in I$, c.a.d à $\Serie(\brick_I)$.
Or ce produit de \textit{séries} est la \textit{fraction rationnelle}
$$
\dfrac{N}{(1-X_1) \cdots (1-X_n)}
\qquad \text{avec} \qquad
N \ = \ \prod\limits_{i \in I} (1-X_i^{d_i}) \ \overset {\rm def}{=}\ \overline{X}^{D(I)}
$$
D'où le résultat puisque, par définition, $N = \scrN_f(\brick_I)$.
\end {proof}

\medskip

Nous fournissons ci-dessous un certain nombre d'applications. Le lecteur s'il
le souhaite pourra en exhiber d'autres.

\subsubsection*{Divers exemples dont la vérification est laissée au lecteur}
  
Ici, l'utilisation de $E \simeq F$ sous entend que $E,F$ sont des
modules $\bbN^n$-gradués et que l'isomorphisme sous-jacent est gradué
de degré 0, de sorte que $E, F$ ont même série de Hilbert-Poincaré.
Par exemple, pour un vecteur de base $e_I$ du complexe de Koszul $\rmK_\sbullet(\uX^D)$:
$$
\bfA[\uX]e_I \simeq X^{D(I)} \bfA[\uX] \qquad \text{via} \qquad
X^\alpha e_I \leftrightarrow X^{D(I)} X^\alpha
$$
Les égalités ou isomorphismes qui suivent sont tous de
même nature: un certain module gradué est isomorphe (ou égal) à une
somme directe de multiples monomiaux de modules-briques.  Cela permet,
en utilisant
$$
\scrN_f(X^\gamma \brick_I) = X^\gamma \scrN_f(\brick_I) = X^\gamma \overline X^{D(I)}
$$
de calculer automatiquement la série d'Hilbert-Poincaré du dit module.

\medskip
$\bullet$
Pour $\#I \ge h$, en rappelant que $\Jex_h[I]$ est le $\bfA$-sous-module de $\Jex_h$ de base
les $X^\alpha \in \Jex_h$ tels que $\DivSeq(X^\alpha) = I$, on dispose
d'une \emph {égalité}:
$$
\Jex_h[I] = X^{D(I)}\, \brick_{\overline I}
$$
La série de $\Jex_h$ obtenue en sommant les séries de $\Jex_h[I]$ sur les $I$ de cardinal $\ge
h$ coïncide avec la première formule de la proposition~\ref{SerieJh}.

\medskip
$\bullet$
Pour $\#J=h$, soit $\Jex'_h[J]$ le $\bfA$-sous-module de $\Jex_h$ de
base les $X^\alpha \in \Jex_h$ tels que l'ensemble des $h$ premiers
éléments de $\DivSeq(X^\alpha)$ soit égal à~$J$. De manière
formelle: $\DivSeq(X^\alpha)[1..h] = J$.  Alors:
$$
\Jex'_h[J] = X^{D(J)}\,\brick_{[1..\max J[ \,\cap\,\overline J}
$$
En sommant les $\Serie\big(\Jex'_h[J]\big)$ sur les $J$ de cardinal $h$, on obtient la formule
dite économique de la proposition~\ref{SerieJh}.

\smallskip
Que dit l'égalité pour $h=0$? Comme l'ensemble vide est la seule
partie de cardinal $0$, on a $\Jex'_h[J] =
\Jex_0 =\bfA[\uX]$. Et par convention, $\max J = 0$ de sorte
que l'indice du module-brique est vide et ce dernier est $\bfA[\uX]$. L'égalité est donc
une tautologie.

\medskip
$\bullet$
Pour $\#I = k$, en notant $\Smac_k[I]$ le sous-module monomial de
$\Smac_k$ de base les $X^\alpha e_I$:
$$
\Smac_k[I] \simeq X^{D(I)}\, \brick_{\min I-1}
\qquad \text{via} \qquad
X^\alpha e_I  \longmapsto X^{D(I)}\,X^\alpha
$$
En utilisant le fait que $\Smac_k$ est la somme directe des
$\big(\Smac_k[I]\big)_{\#I=k}$, on obtient pour sa série la formule de la
  proposition~\ref{SerieSk}.

\smallskip
En passant, que raconte l'isomorphie ci-dessus dans le cas $k=0$? Comme $I
= \emptyset$ et $\min I-1 = (n+1)-1 =n$, l'isomorphie en question
s'écrit:
$$
\Smac_0 \simeq \brick_{\{1...n\}}
$$
Mais en fait, il y a égalité puisque $\brick_{\{1...n\}}$ est le supplémentaire
monomial de $\Jex_{\{1...n\}} = \Jex_1$, i.e. $\Smac_0$. Et tout est
bien qui finit bien.

\medskip
$\bullet$
Pour $\#J = k$, notons $\Mmac_k[J]$ le sous-module monomial de
$\Mmac_k$ de base les $X^\beta e_J \in \Mmac_k$. Alors:
$$
\Mmac_k[J] \simeq X^{D(J)}\, \bigoplus_{i < \min J} X_i^{d_i} \brick_{[1..i[}
$$
l'isomorphisme étant
$$
X^\beta e_J  \longmapsto X^{D(J)}\,X_i^{d_i}\,\dfrac{X^\beta}{X_i^{d_i}}
\qquad \text{où} \qquad
i = \minDiv(X^\beta)
$$
Quid pour $k=0$? Puisque $J = \emptyset$, le module de gauche $\Mmac_k[J]$ n'est
autre que $\Jex_1$. Et à droite, en utilisant $\min\emptyset = n+1$, on obtient
comme module
$$
F := \bigoplus_{i=1}^n X_i^{d_i} \brick_{[1..i[}
$$
En se concentrant un peu, on constate que $F$ est l'idéal monomial de $\bfA[\uX]$
dont les monômes sont ceux qui sont multiples d'un \emph{certain} $X_i^{d_i}$ i.e.
$F = \langle X_1^{d_1}, \dots, X_n^{d_n}\rangle$, bien connu sous le nom de $\Jex_1$.
L'isomorphie est en fait une égalité et on est bien content.

\smallskip

En écrivant que $\Mmac_k$ est la somme directe des $\big(\Mmac_k[J])_{\#J=k}$,
on obtient pour $\scrN_f(\Mmac_k)$ la seconde formule de la proposition~\ref{SerieMk}.

\medskip
$\bullet$
Est-il utile de rappeler la définition de $\Jex_{1\setminus2}^{(i)}$? Nous
pensons que non et nous laissons le soin au lecteur de vérifier
l'égalité:
$$
\Jex_{1\setminus 2}^{(i)} = X_i^{d_i}\,\brick_{\{1...n\}\setminus \{i\}}
$$
On en déduit $\scrN_f\big(\Jex_{1\setminus 2}^{(i)}\big) = X_i^{d_i} \prod\limits_{j \ne i}
(1 - X_j^{d_j})$.

\medskip
$\bullet$
Le retour de $\Jex_h$.
Comme $\Jex_{h+1} \subset \Jex_h$, on peut considérer
le supplémentaire monomial de $\Jex_{h+1}$ dans~$\Jex_h$,
que l'on note $\Jex_{h\setminus h+1}$ et qui est isomorphe, de manière
$\bbN^n$-graduée, au quotient $\Jex_h/\Jex_{h+1}$. On a alors une égalité:
$$
\Jex_{h\setminus h+1} = \bigoplus_{\#I=h} X^{D(I)}\, \brick_{\overline I}
$$
On en déduit, en notant $\scrN$ pour $\scrN_f$: 
$$
\scrN(\Jex_h) - \scrN(\Jex_{h+1})  = s_h \qquad \text{avec} \qquad
s_h = \sum_{\#I=h} X^{D(I)}\,\overline X^{D(\overline I)}
$$
Un petit coup de télescopie \og en avant\fg:
$$
\left\{
\begin {array}{ccccc}
\scrN(\Jex_h) &-& \scrN(\Jex_{h+1})  &=& s_h \\
\scrN(\Jex_{h+1} &-& \scrN(\Jex_{h+2})  &=& s_{h+1} \\
                 &\vdots &&\vdots \\
\scrN(\Jex_n) &-& \scrN(\Jex_{n+1})  &=& s_n \\
\end {array}
\right.
$$
Et comme $\Jex_{n+1} = 0$, en sommant:
$$
\scrN(\Jex_h) = \sum_{\ell \ge h} s_\ell \overset{\rm def.}{=}
\sum_{\#I \ge h} X^{D(I)}\,\overline X^{D(\overline I)}
$$
Et pourquoi pas une télescopie \og en arrière\fg?
$$
\left\{
\begin {array}{ccccc}
\scrN(\Jex_0) &-& \scrN(\Jex_1)  &=& s_0 \\
\scrN(\Jex_{1} &-& \scrN(\Jex_{2})  &=& s_1 \\
                 &\vdots &&\vdots \\
\scrN(\Jex_{h-1}) &-& \scrN(\Jex_h)  &=& s_{h-1} \\
\end {array}
\right.
$$
Mais $\scrN(\Jex_0)=1$ puisque $\Jex_0 = \bfA[\uX]$. Donc, en sommant:
$$
1 - \scrN(\Jex_h) = \sum_{\ell \le h-1} s_\ell 
$$
c.a.d.
$$
\scrN(\Jex_h) = 1 - \sum_{\ell \le h-1} s_\ell =
1- \sum_{\#J \le h-1} X^{D(J)}\,\overline X^{D(\overline J)}
$$
Exercices: combien de formules pour $\scrN(\Jex_h)$? Peut-on les comparer directement?
En $\bbN$-gradué, vérifier:
$$
N_h(t) := \dfrac{s_h(t)}{(1-t)^{n-h}} =
\sum_{\#I=h} t^{d_I} \prod_{j \in \overline I} (1 + t + \cdots + t^{d_j-1})
$$
On définit ainsi un polynôme $N_h(t)$, à coefficients entiers $\ge 0$, tel que:
$$
\Serie_r(\Jex_{h\setminus h+1}) = \dfrac{N_h(t)}{(1-t)^h}
$$
IL est clair que $N_h(1) \ge 1$, en particulier non nul, de sorte que 
cette écriture est l'écriture irréductible de la fraction rationnelle.
Bonus: montrer l'égalité
$$
N_h(1) = \sum_{\#J = n-h} \prod_{j\in J} d_j  \overset{\rm def.}{=}
e_{n-h}(d_1, \cdots, d_n)
$$
où $e_\ell$ est la fonction symétrique élémentaire de degré $\ell$ en $n$
variables.

\subsection{Séries de $\Mmac^{(i)}_{k-1} \simeq \Smac_k^{(i)}$, de $\Jex_h^{(i)}$ etc.
  permettant l'obtention du poids en~$P_i$}

Ici, nous étudions le poids en $P_i$ de certains déterminants associés
au degré~$d$, définis à partir du mécanisme de sélection $\minDiv$. Ce poids en
$P_i$ va s'obtenir comme la dimension de la composante homogène de
degré $d$ d'un certain module gradué dépendant de~$i$ et piloté par le
mécanisme de sélection $\minDiv$.

\medskip

Avant la proposition~\ref{BkdPoidsPiSerie}, nous avons défini, pour $k \ge 1$, le sous-module
monomial $\Mmac^{(i)}_{k-1}$ de $\Mmac_{k-1}$ dont la série est par définition:
$$
\Serie\big(\Mmac_{k-1}^{(i)}\big) \ \overset{\rm def.}{=}\ 
\sum_{d \in \bbN} \dim \Mmac_{k-1,d}^{(i)}\, t^d  \ =\ 
\sum_{d \in \bbN} \poids_{P_i}(\det B_{k,d})\, t^d 
$$
Voir également la proposition \ref{PoidsDetBkdSum} et l'introduction de
ce chapitre. Nous définissons de
manière analogue le sous-module $\Smac^{(i)}_k$ de~$\Smac_k$
de base les monômes extérieurs $X^\alpha e_I \in \Smac_k$ tels que
$\min I = i$. Ces deux sous-modules gradués se correspondent
par l'isomorphisme $\varphi : \Mmac_{k-1} \to \Smac_k$, gradué de
de degré $0$ et en conséquence, ils ont même série d'Hilbert-Poincaré.
\`A noter que nous ne définissons pas $\Smac^{(i)}_k$
pour $k=0$ (d'ailleurs $B_k$ n'est pas défini pour $k=0$).

Dans la proposition~\ref{BkdPoidsPiSerie}, nous avons fourni
la série commune de ces sous-modules pour la raw-graduation.
En utilisant des notations analogues (en particulier la notation
$\Aik$ que nous avons calquée), voici l'énoncé pour la fine-graduation.

\begin{prop}[Série de $\Mmac_{k-1}^{(i)} \simeq \Smac_k^{(i)}$] 
\label{BkdPoidsPiFineSerie}
Soit $k \ge 1$ et $i \in \{1,\dots, n\}$.

Définissons un polynôme $A_{i,k}(\uX) \in \bbZ[\uX]$ de la manière suivante :
$$
A_{i,k}(\uX) \ = \ 
X_i^{d_i} \sum_{\#J = k-1\atop i < \min J} X^{D(J)} 
\ =\ 
\sum_{\# I \, = \, k \atop \min I = i} X^{D(I)} 
$$
Alors:
$$
\scrN_f\big(\Mmac_{k-1}^{(i)}\big) \ =\ 
\scrN_f\big(\Smac_{k}^{(i)}\big) \ =\ 
\overline{X}^{D([1...i[)}\, A_{i,k}(\uX)
$$
La spécialisation $X_i := t$ conduit à
$$
\Aik(t) 
\ = \ 
t^{d_i}\!\!\! \sum_{\#J = k-1\atop i < \min J} \kern -8pt t^{d_J} \ =\ 
\sum_{\#I=k \atop \min I =i} \kern -4pt t^{d_I}
$$
et fournit la série suivante dont le coefficient en $t^d$ est le poids en $P_i$ de $\det B_{k,d}$
$$
\Serie\big(\Mmac_{k-1}^{(i)}\big) = \dfrac{\scrN_r\big(\Mmac_{k-1}^{(i)}\big)}{(1-t)^n}
\qquad \text{avec} \qquad
\scrN_r\big(\Mmac_{k-1}^{(i)}\big) = \Aik(t) \times \prod_{j<i} (1-t^{d_j}) 
$$
Pour $i > n-k+1$, le sous-module $\Mmac_{k-1}^{(i)}$ est réduit à 0 (a
fortiori sa série est nulle), entraînant le fait que $\det B_{k,d}$ ne
dépend pas des $k-1$ derniers polynômes du système~$\uP$.
\end {prop}

\begin {proof}
  
Le fait que les deux sommes définissant $A_{i,k}(\uX)$ sont les mêmes
résulte du \og changement de variables\fg{}
$$
I = \{i\} \sqcup J \ \leftrightarrow \ J = I\setminus\{i\}
$$
sous le couvert de $\min I = i$, ce qui correspond du côté de $J$, à $i < \min J$ ou
encore à $J \subset \{i+1,\dots, n\}$.

\medskip

Par ailleurs, l'existence d'un $X^\beta e_J \in \Mmac_{k-1}^{(i)}$
entraîne, puisque $i = \minDiv(X^\beta)$, l'inclusion $J \subset
\{i+1,\dots, n\}$, a fortiori $\#J \le \#\{i+1,\dots, n\}$ i.e. $k-1 \le
n-i$ i.e. $i \le n-k+1$.  A contrario, si $i > n-k+1$, le sous-module
$\Mmac_{k-1}^{(i)}$ est nul.

\medskip

Enfin, dans l'introduction de ce chapitre, nous avons mentionné la décomposition
$$
\Mmac_{k-1}^{(i)} \ = \ 
\bigoplus_{\#J = k-1\atop i < \min J}  \Mmac_{k-1}^{(i)}[J]
$$
et obtenu
$$
\scrN_f\big(\Mmac_{k-1}^{(i)}[J]\big) =
\overline{X}^{D([1...i[)} X_i^{d_i}X^{D(J)}
$$
En sommant sur les $J$:
$$
\scrN_f\big(\Mmac_{k-1}^{(i)}\big) = 
\overline{X}^{D([1...i[)} X_i^{d_i} \sum_{\#J = k-1\atop i < \min J} \kern -5pt  X^{D(J)}
\ \overset{\rm def}{=}\
\overline{X}^{D([1...i[)}\,\Aik(\uX)
$$
\end {proof}

Nous venons de déterminer la série d'Hilbert-Poincaré du module
$\Mmac_{k-1}^{(i)}$, dont la dimension de la composante homogène de
degré $d$ est le poids en $P_i$ de $\det B_{k,d}$.  Donnons à
présent un résultat analogue pour le déterminant
de $W_{h,d}$ (défini par $\minDiv$) dont le poids en $P_i$ est égal à $\dim \Jex_{h,d}^{(i)}$
(cf. proposition~\ref{PoidsDetW}). Il s'agit donc de fournir la
série du module gradué $\Jex_h^{(i)}$, piloté par le mécanisme de
sélection $\minDiv$. Une partie du travail a été réalisée
en~\ref{J1iJ2iSeries}, en déterminant la série raw-grading de
$\Jex_1^{(i)}$ et $\Jex_2^{(i)}$.

\begin{prop}[Série d'Hilbert-Poincaré de $\Jex_h^{(i)}$] 
\label{SerieJhi}
Soit $h \ge 1$ et $i \in \{1,\dots, n\}$.
On dispose d'une première formule:
$$
\scrN_f\big(\Jex_h^{(i)}\big)
\ = \ 
\sum_{\#I \, \geqslant \, h \atop \min I = i}
X^{D(I)} \overline{X}^{D(\overline I)}
$$
ainsi que d'une seconde dite \og économique\fg{}:
$$
\scrN_f\big(\Jex_h^{(i)}\big)
\ = \ 
\sum_{\#J \, = \, h \atop \min J \, = \, i} 
{X}^{D(J)} \ \overline{X}^{D([1 \dots \max J[ 
\, \cap \, \overline J)}
$$
La spécialisation $X_i := t$ fournit :
$$
\scrN_r\big(\Jex_h^{(i)}\big) = \sum_{\#I \, \geqslant \, h \atop \min I \, = \, i} 
t^{d_I} \prod_{j \notin I} (1-t^{d_j})
\ = \ 
\sum_{\#J \, = \, h \atop \min J \, = \, i}
t^{d_J} \prod_{j \notin J \atop j < \max J} (1- t^{d_j})
$$
\end{prop}

\begin {proof}\leavevmode

Pour une partie $I$ vérifiant $\#I \ge h$ et $\min I = i$, notons
$\Jex_h^{(i)}[I]$ le sous-module monomial de $\Jex_h^{(i)}$ de base
les $X^\alpha \in \Jex_h^{(i)}$ tels que $\DivSeq(X^\alpha) = I$ et
$\minDiv(X^\alpha) = i$. Le lecteur, en utilisant la notion
de module-brique de la section~\ref{SectionBriques}, vérifiera l'égalité:
$$
\Jex_h^{(i)}[I] = X^{D(I)} \brick_{\overline I}
$$
D'où  $\scrN_f\big(\Jex_h^{(i)}[I]) = X^{D(I)} \overline{X}^{D(\overline I)}$
et par sommation sur les $I$, l'obtention de la première formule.

\medskip

Pour la deuxième formule, on va utiliser le point iii) du lemme
économique~\ref{SecondEcoLemma}.  A cet effet,
on pose $F = \{i+1..n\}$, $G = \{1..i-1\}$ et on écrit toute partie $I$ ci-dessus
sous la forme $I = \{i\} \sqcup I'$ de sorte que
$$
I' \subset F,
\qquad
\#I' \ge h-1,
\qquad
\overline I = G \sqcup (F \setminus I')
$$
Ceci permet, dans la somme de la première formule, de \og sortir le morceau constant\fg{}:
$$
\Pi := X_i^{d_i} \overline X^{D(G)}
$$
De manière précise, $\scrN_f\big(\Jex_h^{(i)}\big) = \Pi \times s$ avec
$$
s = \sum_{I' \subset F\atop \#I' \geqslant\, h-1} X^{D(I')}
\overline{X}^{D(F \setminus I')}
\quad \overset{\rm eco.}{=} \quad
\sum_{J' \subset F\atop \#J'=h-1} X^{D(J')} \overline{X}^{D((F \setminus J') \cap F_{<\max J'})}
$$
On fait \og rentrer $\Pi$ dans la somme $s$\fg{} (formule de droite) en posant
$J = \{i\} \sqcup J'$. On a d'une part $\max J = \max J'$. Et d'autre part
l'égalité a priori rebutante mais en fait anodine:
$$
G \sqcup \big((F \setminus J') \cap F_{<\max J'}\big) =
[1...\max(J)[\ \cap\ \overline J
$$
ce qui conduit à la deuxième formule.
\end {proof}

\medskip

Pour $h=1$, on obtient 
$$
\scrN_r\big(\Jex_1^{(i)}\big) \ = \ 
t^{d_i} \prod\limits_{j < i} (1- t^{d_j})
$$
formule qui coïncide avec celle fournie en~\ref{J1iJ2iSeries}.
Qu'en est-il pour $h=2$? Ici, en utilisant la formule économique,
on dispose d'une somme indexée par les parties $J$ de cardinal 2 vérifiant
$i = \min J$. On a donc $J = \{i,i'\}$ avec $i' > i$ et on peut ré-écrire
la somme en l'indexant par les $i' > i$:
$$
\scrN_r\big(\Jex_2^{(i)}\big) =
t^{d_i} \times \sum_{i'>i} t^{d_{i'}} \prod_{j\ne i \atop j< i'} (1-t^{d_j})
\leqno (A)
$$
Versus la formule donnée en \ref{J1iJ2iSeries} comme produit de 3 termes:
$$
\scrN_r\big(\Jex_2^{(i)}\big) =
t^{d_i} \times \prod_{j<i} (1-t^{d_j}) \times \Big(1 - \prod_{j>i} (1-t^{d_j}) \Big)
\leqno (B)
$$
Pour vérifier l'adéquation entre $A$ et $B$, on écrit, dans la somme
indexée par $i'$ définissant $A$, le produit interne comme produit de
deux sous-produits, le premier $\Pi_1$ étant indépendant de $i'$
$$
\prod_{j\ne i \atop j< i'} (1-t^{d_j}) =  \Pi_1 \times \prod_{i < j < i'} (1-t^{d_j})
\qquad \text{avec} \qquad
\Pi_1 = \prod_{j<i} (1-t^{d_j})
$$
Ceci fait apparaître des quantités communes à $A,B$ que l'on peut écrire:
$$
A = t^{d_i} \times \Pi_1 \times A',
\qquad\qquad
B = t^{d_i} \times \Pi_1 \times B'
$$
avec
$$
A' = \sum_{i'>i} t^{d_{i'}} \prod_{i<j<i'} (1-t^{d_j}),
\qquad\qquad
B' = 1 - \prod_{j>i} (1-t^{d_j})
$$
Reste à voir que $A'=B'$.  Tout se passe maintenant dans l'ensemble
fini totalement ordonné $F$ constitué des indices $> i$. En définissant
$f : F \to \bbZ[t]$ par $f(j) = 1 - t^{d_j}$ et en utilisant les
notations de la section~\ref{BonusFormulesEco}, l'égalité $A'\overset{?}{=}B'$ s'écrit
$$
\sum_{i' \in F} \big(1-f(i')\big) \pi_0(F_{<i'}) \overset{?}{=} 1 - \pi_0(F) 
$$
Mais comme $\big(1-f(i')\big)\pi_0(F_{<i'}) = \pi_0(F_{<i'}) - \pi_0(F_{\le i'})$,
la somme de gauche est télescopique et se réduit au membre droit. Ouf!


Il n'y a rien de vraiment nouveau dans la proposition qui vient. La
détermination de $\dim \Jex_{1\setminus 2, d}^{(i)}$ pour~$d \ge
\delta$ a été réalisée dès le chapitre~\ref{ObjetsSuiteP} (cf. la
proposition~\ref {LemmeDenombrementJ1moins2id}), sans utilisation de
série, résultat important utilisé à maintes reprises.  Nous avons
également fourni la série $\bbN$-graduée de $\Jex_{1\setminus 2}^{(i)}$ en
\ref{SerieJ1minusJ2i} et celle-ci est par exemple intervenue dans la preuve du
lemme~\ref{varpiresPoids}. Ici, nous apportons un petit complément dans
l'énoncé et, dans la preuve, une légère variation concernant la
détermination de ses coefficients.

Quant à la série $\bbN^n$-graduée, elle figure parmi les exemples
de modules-briques de la section~\ref{SectionBriques} et redonne
gratuitement la série $\bbN$-graduée.

Il faut noter que $\Jex_{1\setminus 2}^{(i)}$ est un module monomial
intrinsèque, indépendant de tout mécanisme de sélection, contrairement
aux modules intervenant précédemment dans cette section, pilotés par
le mécanisme de sélection~$\minDiv$.  Ceci explique l'importance de sa série, dont
le coefficient en $t^d$ fournit le poids en~$P_i$ de
$\Delta_{1,d}(\uP)$ (autre objet intrinsèque), comme nous l'avons vu
en~\ref{poidsDelta1}.

\begin{prop}[Série de $\Jex_{1\setminus 2}^{(i)}$ et ses coefficients]
Pour $i \in \{1,\dots, n\}$:
$$
\scrN_f\big(\Jex_{1\setminus 2}^{(i)}\big)  
\ = \ 
X_i^{d_i} 
\prod\limits_{j \neq i} (1- X_j^{d_j}) 
\qquad \text{ et } \qquad 
\scrN_r\big(\Jex_{1\setminus 2}^{(i)}\big)   \ = \ 
t^{d_i} \ \prod_{j \neq i} (1-t^{d_j}) 
$$
La suite des coefficients de la série de $\Jex_{1\setminus 2}^{(i)}$
est croissante (au sens large), stationnaire à partir du
degré~$\delta+1$.  Le coefficient en $t^d$ pour $d \geqslant \delta+1$
est $\widehat d_i := \prod_{j \neq i} d_j$, et pour $d = \delta$, il
vaut $\widehat d_i - 1$.
\end{prop}

\begin{proof}

En conservant les notations de la proposition \ref{SerieJ1minusJ2i},
nous faisons intervenir le \emph{polynôme} $U_i(t)$ dans le
développement de la série de $\Jex_{1\setminus2}^{(i)}$:
$$
\Serie\big(\Jex_{1\setminus2}^{(i)}\big) =
\dfrac{U_i(t)}{1-t} 
\qquad \text{où } \qquad
U_i(t) \ = \ 
t^{d_i} \ \prod_{j \neq i} \dfrac{1-t^{d_j}}{1-t}
\ = \ 
t^{d_i} \, \prod_{j \neq i} 
\big(1+t+\cdots + t^{d_j-1} \big)
$$
C'est l'occasion, pour un polynôme $P(t)$, de rappeler le développement en série de
$P(t) / (1-t)$. Par exemple, pour $P(t) = a + bt + ct^2$, on a
facilement:
$$
(a + bt + ct^2)(1 + t + t^2 + \cdots) =
a + (a+b)t + (a+b+c)(t^2 + t^3 +  \cdots)
$$
Plus généralement, pour un polynôme de degré $r$:
$$
P(t) = a_0 + a_1t + \cdots + a_rt^r
$$
le développement en série de $P(t)/(1-t)$ est stationnaire à partir du degré $r$
$$
\dfrac{P(t)}{1-t} = s_0 + s_1t + \cdots + s_{r-1}t^{r-1} + s_r(t^r + t^{r+1} + \cdots)
\qquad \text{où} \qquad
s_\ell = \sum_{k=0}^\ell a_k
$$
Il est souvent utile d'avoir remarqué que $s_r = P(1)$, $s_{r-1} =
P(1)-a_r$ et $s_0 = a_0$.  Lorsque $P$ est à coefficients dans $\bbN$,
ce développement montre que la suite des coefficients de la série est
croissante au sens large.

\medskip

Retour au polynôme $U_i$. Il est unitaire, à coefficients entiers $\ge
0$, de degré $d_i + \sum_{j \neq i} (d_j-1) = \delta +1$ et $U_i(1) =
\prod_{j \neq i} d_j = \widehat{d_i}$. D'où le résultat concernant les
coefficients de la série de $\Jex_{1\setminus 2}^{(i)}$ en degré $\ge
\delta$.

\smallskip

En ce qui concerne la croissance de la suite
$\big(\dim\Jex_{1\setminus2, d}^{(i)}\big)_{d \ge 0}$, il y a un
autre argument, de nature structurelle, qui réside dans l'injection:
$$
\Jex_{1\setminus2,d}^{(i)} \quad\hookrightarrow\quad \Jex_{1\setminus2,d+1}^{(i)},
\qquad
X^\alpha \mapsto X_i X^\alpha
$$
Et le stationnement de la suite s'explique par le fait que cette
injection est une \emph{bijection} pour $d \ge \delta+1$.  On peut le
vérifier en constatant, pour un $X^\beta \in
\Jex_{1\setminus2,d+1}^{(i)}$, en n'importe quel degré $d$, que l'on a:
$$
\beta_i \ge (d-\delta) + d_i
$$
Si bien que pour $d \ge \delta+1$, en posant $X^\alpha :=
X^\beta/X_i$, on a $\alpha_i \ge d_i$ donc $X^\alpha \in
\Jex_{1\setminus2,d}^{(i)}$.

\smallskip

Enfin, l'égalité dimensionnelle $\dim\Jex_{1\setminus2,\delta+1}^{(i)}
= 1+\dim\Jex_{1\setminus2,\delta}^{(i)}$ peut s'expliquer par celle
des modules:
$$
\Jex_{1\setminus2,\delta+1}^{(i)} = X_i\,\big(\bfA X^\emouton \oplus \Jex_{1\setminus2,\delta}^{(i)}\big)
$$
\end{proof}

\medskip
En guise d'exercice, le lecteur pourra prouver l'égalité:
$$
\dim\Jex_{1\setminus2,\delta-1}^{(i)} = \widehat d_i - 1 - \#I_i
\qquad \text{où} \qquad
I_i = \big\{j \in \{1..n\} \setminus \{i\} \mid d_j\ge 2 \big\}
$$
Une manière d'y parvenir consiste à montrer:
$$
\Jex_{1\setminus2,\delta}^{(i)} = X_i\,E_i
\qquad \text{avec} \qquad
E_i = \Jex_{1\setminus2,\delta-1}^{(i)} \ \oplus \ \bigoplus_{j\in I_i}\bfA(X^\emouton/X_j)
$$
Indication: pour $X^\alpha \in \Jex_{1\setminus2,\delta}^{(i)}$, étudier $X^\alpha/X_i$.

\newpage

\centerline{
\fbox{\parbox{0.95\linewidth}{
J'ai déporté les 2 dernières sections du chapitre Séries consacrées aux relations
binomiales entre séries dans un fichier {\tt 19-SouvenirsBinomiaux.tex} chez moi.
On verra plus tard si on peut en faire quelque chose
}}}

\appendix
\cleardoublepage


\section{Une preuve homologique du théorème de Wiebe}
\label{WiebeProof}

Voici une reformulation du théorème de Wiebe, qui a l'avantage 
de s'énoncer en une seule \og phrase \fg{} : il s'agit de montrer qu'un 
certain morphisme est bijectif.

\begin{quote}
\textit{
Soient $\ux = (x_1,\dots,x_n)$ et $\ua = (a_1,\dots,a_n)$ deux suites
de même longueur d'un anneau commutatif $\bbA$ telles que $\ideala
\subset \idealx$, inclusion certifiée par une matrice de
$\bbM_n(\bbA)$ dont on note~$\nabla$ le déterminant.
Notons $(\ua:\nabla)$ l'idéal transporteur $(\ideala:\langle\nabla\rangle)$ 
et de la même manière $(\ua:\ux)$ pour $(\ideala:\idealx)$.
Si la suite $\ua$ est complètement sécante, alors
on a l'égalité d'idéaux de $\bbA$:
$$
(\ua : \nabla) \ = \ \idealx
\qquad \text{et} \qquad 
(\ua : \ux) \ = \  \langle \nabla,\, \ua \rangle 
$$
De manière équivalente, en désignant par $\Ann(\ux\,;\bbA/\ideala)$
l'idéal de $\bbA/\ideala$ annulateur de $\ux$, 
la multiplication par $\nabla$ induit un isomorphisme
$$
\xymatrix @M=0.4pc @C=3pc{
\bbA/\idealx \ar[r]^-{\times \nabla} & 
\Ann(\ux \,; \bbA/\ideala)
}
$$
De manière plus précise, l'injectivité de $\times\nabla$ équivaut à l'égalité
$(\ua : \nabla)=\idealx$ et sa surjectivité à l'égalité
$(\ua:\ux) = \langle\nabla,\,\ua\rangle$.
}
\end{quote}

\index{théorème!de Wiebe}

\subsection {Le théorème du bicomplexe pour les bébés}

Le théorème du bicomplexe (de bébés) que nous allons utiliser s'inspire de
Matsumura \cite[Appendix~B (Some homological
  algebra)]{Matsumura}. Tout comme lui, nous allons nous placer dans
un contexte très particulier. Voici ce qu'il écrit à propos
de ce cas particulier: ``The basic technique for studying the homology
of double complexes is spectral sequences, but we leave this to
specialist texts, and only consider here the extreme cases which we
will use later''. Nous verrons que ``extreme cases'' est relatif à
une hypothèse ad-hoc d'exactitude dans le bicomplexe (cf. l'énoncé
du prochain théorème).

\smallskip

Tous les détails ne seront pas fournis ici. Pour plus de précisions,
nous renvoyons à l'appendice de Matsumura ou bien à la thèse de
C. Tête en \cite[chap. IV]{Tete}.  Nous allons considérer un
bicomplexe descendant $(C_{\sbullet,\sbullet}, d', d'')$ concentré
dans le premier quadrant \idest{} $C_{\sbullet,\sbullet} =
(C_{p,q})_{p,q}$ avec $C_{p,q} = 0$ si $p < 0$ ou $q < 0$:
$$
d': C_{p,q} \to C_{p-1,q},\qquad d'': C_{p,q} \to C_{p,q-1},\qquad
d' \circ d'' = d'' \circ d'
$$
Nous notons $(T_\sbullet, d)$ son complexe total:
$$
T_n = \bigoplus_{p+q=n} C_{p,q}, \qquad\qquad
d(c_{p,q}) = d'(c_{p,q}) + (-1)^p d''(c_{p,q})
$$
On lui associe son $X_\sbullet$-complexe (resp. $Y_\sbullet$-complexe)
qui a la vocation de boucher l'axe des abscisses (resp. ordonnées) de manière exacte 
via le conoyau des flèches verticales (resp. horizontales) au bord de l'axe.
$$
\vcenter{\hbox{
\begin {tikzpicture} [inner sep = 1pt]
\coordinate (O) at (0,0) ;
\node at (2.4,1.7) {$C_{p,q}$} ;
\path (2.4, 1.4) edge [->] (2.4, 0.4) ;  
\path (2.5, 1.1) [right, ->] node {$d''$} ;
\path (1.9,1.7) edge [->] (0.4,1.7) ;  
\path (1.2,1.77) [above] node {$d'$} ;
\path (O) edge [->] (2.5,0) ; 
\path (0.1,-0.2) edge [below, <-] node {$X_\sbullet$-complexe} (2.2,-0.2) ;
\path (O) edge [->] (0,2.5) ; 
\path (0.1, 0.2) edge [rotate = 90, <-, above, midway] node [rotate=90] {$Y_\sbullet$-complexe} (2.1, 0.2) ;
\end {tikzpicture}
}}
\qquad
X_p = \Coker(d'' : C_{p,1} \to C_{p,0}),
\qquad
Y_q = \Coker(d' : C_{1,q} \to C_{0,q})
$$
On dispose ainsi de deux morphismes de complexes surjectifs (donnés par les projections canoniques) 
qui donnent naissance à des morphismes en homologie
$$
\xymatrix@M=0.4pc{
Y_\sbullet & T_\sbullet \ar@{->>}[l] \ar@{->>}[d] \\
& X_\sbullet
}
\hspace{4cm}
\xymatrix@M=0.4pc{
\rmH(Y_\sbullet) & \rmH(T_\sbullet) \ar[l] \ar[d] \\
& \rmH(X_\sbullet)
}
$$
Voici, sans preuve, le théorème que nous allons utiliser.

\begin{theo}[Du bicomplexe pour les bébés] \leavevmode
\label{BabyBicomplexTh}
  
\begin {enumerate}[\rm i)]  
\item
Si les colonnes sont exactes, sauf éventuellement sur l'axe des
abscisses, alors $\rmH(T_\sbullet) \rightarrow \rmH(X_\sbullet)$ est
un isomorphisme.

\item
Si les lignes sont exactes, sauf éventuellement sur l'axe des
ordonnées, alors $\rmH(T_\sbullet) \rightarrow \rmH(Y_\sbullet)$ est
un isomorphisme.
\end {enumerate}  
\end{theo}

\index{théorème!du bicomplexe}

Pour une preuve, voir \cite[Theorem B1]{Matsumura} ou \cite[chap. IV,
section 3, propositions 3.1 \& 3.2]{Tete} chez qui l'exactitude
demandée est \og plus localisée\fg{}. Par exemple, si les colonnes
sont exactes en les points de la diagonale $n$, sauf éventuellement
sur l'axe des $x$, alors $\rmH_n(T_\sbullet) \to \rmH_n(X_\sbullet)$
est surjectif.  Tandis qu'en cas d'exactitude des colonnes en les
points de la diagonale $n+1$, sauf éventuellement sur l'axe des $x$,
le morphisme $\rmH_n(T_\sbullet) \to \rmH_n(X_\sbullet)$ est injectif.

\medskip

\noindent
\begin{minipage}[c]{0.7\textwidth}
Supposons la condition i) vérifiée de sorte que $\rmH(T_\sbullet)
\rightarrow \rmH(X_\sbullet)$ est un isomorphisme.  La composition de
$\rmH(Y_\sbullet) \leftarrow \rmH(T_\sbullet)$ avec l'isomorphisme
inverse donne naissance à une flèche pointillée comme ci-contre.
Si de plus la condition ii) est vérifiée alors cette flèche pointillée
est un isomorphisme. C'est cet isomorphisme en pointillé, très précis,
qui va intervenir dans la suite. Enoncer seulement $\rmH(X_\sbullet)
\simeq \rmH(Y_\sbullet)$ représente une perte considérable d'information.
\end{minipage}
\begin{minipage}[c]{0.3\textwidth}
$$
\xymatrix@M=0.4pc{
\rmH(Y_\sbullet) & \rmH(T_\sbullet) \ar[l] \ar[d]|-{\simeq} \\
& \rmH(X_\sbullet) \ar@{-->}[ul]\ar@/_10pt/[u]
}
$$
\end{minipage}

\subsubsection*{Fonctorialité en $C_{\sbullet,\sbullet}$ des complexes $T_\sbullet$, $X_\sbullet$ et $Y_\sbullet$}

\noindent
\begin{minipage}[c]{0.7\textwidth}
Un morphisme 
$\xymatrix{
C'_{\sbullet, \sbullet} \ar@[blue][r] & C_{\sbullet, \sbullet}
}$ 
entre deux bicomplexes induit un morphisme au niveau des complexes totaux 
ainsi qu'au niveau des complexes $X_\sbullet$ et $Y_\sbullet$ comme ci-contre.
Si les colonnes des deux bicomplexes sont exactes sauf sur l'axe des abscisses,
on a alors un diagramme commutatif 
$$
\xymatrix@M=0.5pc{ 
\rmH(X'_\sbullet) \ar@[blue][d] \ar@{-->}[r] 
& \rmH(Y'_\sbullet) \ar@[blue][d] \\
\rmH(X_\sbullet)  \ar@{-->}[r] & \rmH(Y_\sbullet)
}
$$
\end{minipage}
\begin{minipage}[c]{0.3\textwidth}
$$
\xymatrix @M=0.5pc @R=1.7pc{
Y'_\sbullet & & T'_\sbullet \ar@{->>}[ll] \ar@{->>}[dl] \\
& X'_\sbullet & \\
Y_\sbullet \ar@{<-}@[blue][uu] & & 
T_\sbullet \ar@{->>}[ll] \ar@{->>}[dl] \ar@{<-}@[blue][uu] \\
& X_\sbullet \ar@{<-}@[blue][uu] & \\
}
$$
\end{minipage}

\noindent
\fbox{\parbox{0.95\linewidth}{%
{\bf Précision}
\label{EncadrePrecision}

Dans cette annexe, nous allons faire un usage intensif du complexe de Koszul montant
d'une suite de scalaires et nous utiliserons un certain nombre d'isomorphismes dits
canoniques dans la catégorie des algèbres graduées alternées ou des complexes (différentiels).
En particulier, pour deux suites de scalaires $\ub$ et~$\uc$ de longueurs
quelconques, l'isomorphisme (canonique) suivant va jouer un rôle fondamental:
$$
\rmK^\sbullet(\ub,\uc) \overset{\rm can.}{\simeq}\rmK^\sbullet(\ub)\otimes\rmK^\sbullet(\uc)
\leqno (\heartsuit)
$$
En fin d'annexe, nous avons prévu en section~\ref{AlgGradAlt} une
description relativement basique de ces notions, parfois informelle, pensant qu'elle
pourra être utile aux lecteurs (et aux auteurs!).  Dans notre traitement
exploitant ces notions, nous n'hésiterons pas à fournir un
certain nombre de détails.  Nous pensons que les outils homologiques utilisés sont
peut-être plus importants que le théorème de Wiebe proprement dit.
Comme complément, nous suggérons la lecture de \cite [section
  I.6. The Koszul Complex]{BrunsHerzog} dans l'ouvrage de Bruns \& Herzog
où l'énoncé de l'isomorphie $(\heartsuit)$ figure en proposition~1.6.6.
}}

\subsubsection*{A propos de l'abondance à venir des isomorphismes qualifiés de canoniques}

C'est en fait un des problèmes principaux que nous allons devoir affronter!
Dans notre histoire, il y a des identifications inoffensives et
d'autres qui le sont moins. Pour une suite de scalaires $\ua$ de
longueur~$n$, écrire $\rmK^0(\ua) = \bbA$ est licite car c'est une
vraie égalité.  En revanche écrire $\rmK^n(\ua) \overset{\rm
  can.}{\simeq} \bbA$ puis faire le glissement $\rmK^n(\ua) = \bbA$
sous prétexte qu'il s'agit d'un isomorphisme canonique masque une
certaine vérité. Idem pour $\rmH^n(\ua) \simeq \bbA/\ideala$ et le
glissement $\rmH^n(\ua) = \bbA/\ideala$.  Il est parfois bien
préférable d'écrire:
$$
\rmK^n(\ua) = \BW^n(\bbA^n) 
\qquad \text{et pourquoi pas} \qquad
\rmK^n(\ua) = \BW^n\rmK^1(\ua) 
$$
L'écriture de droite permet d'attribuer la base canonique
$(e_1, \dots, e_n)$ de $\bbA^n$ à $\rmK^1(\ua)$ et d'écrire:
$$
\rmK^n(\ua) = \BW^n\rmK^1(\ua) = \bbA.e_1\wedge\cdots \wedge e_n
$$
On ne peut s'empêcher de citer J.-P. Jouanolou, dans Aspects invariants de
l'élimination~\cite{J5}, après l'exemple~2.5.4, bas de la page~43: {\sl Bien
entendu, si j'insiste sur ces questions en substances faciles, c'est
pour diminuer au maximum dans les applications le nombre
d'isomorphismes canoniques, qui sont, comme on le sait, la plaie de
l'algèbre homologique.}

Dans le même ordre d'idées, ce même auteur, dans cet article,
entre la proposition~3.6.2.5 et sa preuve (p.~133), \og avoue\fg{}:
{\sl On va montrer qu'il existe $\varepsilon = \pm 1$ tel que ...
Il est fort probable que $\varepsilon = +1$ mais la manière
détournée de prouver (3.6.2.5) ne permet pas de l'affirmer, et
l'accumulation d'isomorphismes ``canoniques'' mais peu
explicites dans la théorie de la dualité globale est peu
propice à une vérification.}

\bigskip

Nous donnons un premier aperçu de l'isomorphisme fondamental figurant dans l'encadré,
cas particulier d'un isomorphisme canonique très général $\BW^\sbullet(M\oplus N) \simeq
\BW^\sbullet(M) \gotimes \BW^\sbullet(N)$ entre algèbres graduées alternées
(très général car $M,N$ sont des modules quelconques). Nous en reparlerons
en section~\ref{AlgGradAlt}.

\subsubsection*{Les termes des complexes dans $(\heartsuit)$}

L'isomorphisme $(\heartsuit)$ réalise au niveau des termes des complexes:
$$
\rmK^k(\ub,\uc) \simeq \bigoplus_{i+j=k} \rmK^i(\ub) \otimes \rmK^j(\uc)
$$
Afin de décrire cette correspondance, nous écrivons:
$$
\rmK^i(\ub) = \BW^i\rmK^1(\ub),
\qquad
\rmK^j(\uc) = \BW^j\rmK^1(\uc),
\qquad
\rmK^k(\ub,\uc) = \BW^k\big(\rmK^1(\ub)\oplus\rmK^1(\uc)\big)
$$
Soit $\bfu\in \rmK^i(\ub)$. Via (la puissance extérieure $\BW^i$ de) l'injection canonique $\rmK^1(\ub)
\to\rmK^1(\ub)\oplus\rmK^1(\uc)$, nous voyons $\bfu$ dans
$\BW^i(\rmK^1(\ub)\oplus\rmK^1(\uc)\big)$.  Idem pour $\bfv\in
\rmK^j(\uc)$ que nous voyons dans
$\BW^j(\rmK^1(\ub)\oplus\rmK^1(\uc)\big)$. En notant $k=i+j$, leur
produit extérieur $\bfu\wedge\bfv$ est dans
$\BW^k(\rmK^1(\ub)\oplus\rmK^1(\uc)\big)$ et la correspondance est la
suivante
$$
\bfu\wedge\bfv  \quad\longleftrightarrow\quad \bfu\otimes\bfv
$$

\subsubsection*{Les différentielles des complexes dans $(\heartsuit)$}

\index{produit tensoriel!de complexes}%

Le produit tensoriel $B^\sbullet \otimes C^\sbullet$ de
deux complexes montants $(B^\sbullet,\delta^B)$ et
$(C^\sbullet,\delta^C)$ est le complexe montant équipé de la
différentielle $\delta$ définie sur $B^i \otimes C^j$ par:
$$
\delta_{|B^i\otimes C^j} = \delta^B_{|B^i}\otimes\Id_{C^j} + (-1)^i\,\Id_{B^i}\otimes\delta^C_{|C^j}
$$
On suppose $\ub$ de longueur $m$, $\uc$ de longueur $n$, 
$\ub = (b_1, \dots, b_m)$, $\uc = (c_1, \dots, c_n)$.
On note $(e_1,\dots,e_m)$ la base canonique de $\rmK^1(\ub)$
et  $(f_1,\dots,f_n)$ celle de $\rmK^1(\uc)$.
Les termes de $\rmK^\sbullet(\ub,\uc)$ sont les
$$
\BW^k\big(\rmK^1(\ub) \oplus \rmK^1(\uc)\big)
$$
Les différentielles des complexes de Koszul montants
$\rmK^\sbullet(\ub), \rmK^\sbullet(\uc), \rmK^\sbullet(\ub,\uc)$
sont définies
à l'aide du produit extérieur par les vecteurs 
$$
b = \sum_{i=1}^m b_ie_i \in \rmK^1(\ub), \qquad\quad
c = \sum_{j=1}^n c_jf_j \in \rmK^1(\uc), \qquad\quad
b+c \in \rmK^1(\ub)\oplus \rmK^1(\uc)
$$
ceux-ci étant placés \emph {à gauche.}\footnote{Curieusement, dans \cite[Exercice 1.6.28]{BrunsHerzog},
les auteurs définissent, à partir d'un $R$-module $L'$ et $x \in L'$,  un complexe montant
$\BW^\sbullet(L')$ ayant pour différentielle   ``the \emph{right} multiplication by $x$''.
Ils demandent ensuite de montrer, lorsque $L' = (R^n)^\star$, que ce complexe est
isomorphe au complexe de Koszul montant de~$x$, défini par dualité à partir du complexe de Koszul
descendant de $x$. Ce dernier, défini pour tout module $L$ et toute forme linéaire $\mu\in L^\star$,
a pour différentielle de degré~$-1$ le produit intérieur droit $\sbullet\intd\mu$, unique
anti-dérivation à gauche de degré~$-1$ prolongeant $\mu$.}
Par exemple, en ce qui concerne $\rmK^\sbullet(\ub)$, nous avons choisi pour différentielle (montante)
de degré $+1$:
$$
\delta^\ub = b \wedge \sbullet
$$
\label{AdequationDifferentielles}
Nous reportons en section~\ref{AlgGradAlt} la vérification de
l'adéquation entre la différentielle du produit tensoriel, définie sur
$\rmK^i(\ub) \otimes \rmK^j(\uc)$ par:
$$
\delta^\ub \otimes \Id + (-1)^i \Id\otimes \delta^\uc
$$
et la différentielle $\delta^{\ub,\uc}$ de $\rmK^{\ub,\uc}$ donnée par $(b+c)\wedge\sbullet$

\subsection {Le bicomplexe Koszul-Koszul
  $\rmK^{n-\sbullet}(\protect\ua)\otimes\rmK^{n-\sbullet}(\protect\ux)$ : étape I}

Les suites de scalaires $\ua$ et $\ux$ sont de même longueur~$n$ et
pour l'instant, on ne suppose rien d'autre. I.e. on n'impose pas la
relation d'inclusion $\ideala \subset \idealx$, celle-ci interviendra
en temps utile.  Considérons le bicomplexe bâti à partir du complexe
de Koszul montant de $\ua$ et de celui de $\ux$, que l'on installe, de
manière à être en conformité avec la section précédente, dans le
premier quadrant:
$$
\vcenter{\hbox{%
\begin {tikzpicture} [inner sep = 1pt]
\coordinate (O) at (0,0) ;
\node at (2.4,1.7) {$C_{p,q}$} ;
\path (2.4, 1.4) edge [->] (2.4, 0.6) ;  
\path (2.5, 1.1) [right, ->] node {$\ux$-Koszul} ;
\path (1.9,1.7) edge [->] (0.5,1.7) ;  
\path (1.2,1.77) [above] node {$\ua$-Koszul} ;
\path (O) edge [->] (2.5,0) ; 
\path (0.1,-0.2) edge [below, <-] node {$X_\sbullet$-complexe} (2.2,-0.2) ;
\path (O) edge [->] (0,2.5) ; 
\path (0.1, 0.2) edge [rotate = 90, <-, above, midway] node [rotate=90] {$Y_\sbullet$-complexe} (2.1, 0.2) ;
\end {tikzpicture}
}}
\qquad \qquad \qquad
\begin {array}{l}
C_{p,q} = \rmK^{n-p}(\ua\,;\bbA) \otimes \rmK^{n-q}(\ux\,;\bbA)
\\[2mm]
d' = \delta^\ua\otimes\Id
\\[2mm]
d'' = \Id\otimes \delta^\ux
\\[2mm]
(d' \circ d'' = d''\circ d')
\end{array}  
$$
Dans la suite, il faudra veiller au reversement $n-\sbullet$ que l'on
observe dans la définition de $C_{\sbullet,\sbullet}$, renversement qui fait des
complexes $X_\sbullet$ et $Y_\sbullet$ des complexes
\emph{descendants} \og plaqués\fg{} aux axes comme l'indique le
dessin.

\medskip

Par définition de $X_p=\Coker(d'': C_{p,1} \to C_{p,0})$:
$$
X_p = \Coker\big(\rmK^{n-p}(\ua) \otimes \rmK^{n-1}(\ux)
\xrightarrow{\ \Id\otimes\delta^\ux\ }
\rmK^{n-p}(\ua) \otimes \rmK^{n}(\ux)\big)
$$
Montrons que $X_p \overset{\rm can.}{\simeq}
\rmK^{n-p}\big(\ua\,;\bbA/\idealx\big)$.
Par définition $\rmK^{n}(\ux) = \BW^n(\bbA^n) \simeq \bbA$
et $\delta^\ux : \rmK^{n-1}(\ux) \to
\rmK^{n}(\ux)$ a pour image $\idealx\rmK^{n}(\ux)$.
En conséquence
$$
\Im\Big(\rmK^{n-p}(\ua) \otimes \rmK^{n-1}(\ux) \xrightarrow{\ \Id\otimes\delta^\ux\ }
\rmK^{n-p}(\ua) \otimes \rmK^{n}(\ux)\Big) =
\idealx\,\rmK^{n-p}(\ua) \otimes \rmK^{n}(\ux) \simeq
\idealx\,\rmK^{n-p}(\ua)
$$
de sorte que le conoyau $X_p$ \emph{s'identifie} à
$\rmK^{n-p}(\ua)/\idealx\rmK^{n-p}(\ua) = \rmK^{n-p}(\ua)
\otimes\bbA/\idealx = \rmK^{n-p}(\ua; \bbA/\idealx)$.  Il est
important de signaler que, 2 lignes plus haut, nous avons identifié
$\rmK^n(\ux) = \BW^n(\bbA^n)$ à $\bbA$. Il faudra se souvenir qu'en
notant \boxed{(f_1,\dots,f_n)} la base canonique 
de \boxed{\bbA^n=\rmK^1(\ux)}, nous avons
$$
\rmK^n(\ux) = \BW^n \rmK^1(\ux) = \bbA.(f_1\wedge\cdots\wedge f_n)
$$
De manière analogue, en identifiant $\rmK^n(\ua)$ à
$\bbA$, on obtient l'identification de $Y_p$ à
$\rmK^{n-p}\big(\ux\,;\bbA/\ideala\big)$. Nous aurons besoin d'un \emph
{autre nom} pour désigner la base canonique 
de \boxed{\bbA^n = \rmK^1(\ua)}. En choisissant de la nommer
\boxed{(e_1,\dots,e_n)}, nous avons
$$
\rmK^n(\ua) = \BW^n \rmK^1(\ua) = \bbA.(e_1 \wedge\cdots\wedge e_n)
$$
En guise de résumé, on dispose donc, pour les projections $T_\sbullet
\to X_\sbullet$ et $T_\sbullet \to Y_\sbullet$, du schéma suivant dans
lequel on entrevoit le passage au quotient de $C_{k,0}$ à $X_k$
via l'idéal $\idealx$ (resp. de $C_{0,k}$ à $Y_k$ via l'idéal $\ideala$):
$$
\xymatrix @R=1.0cm @C=1.0cm{
Y_k = C_{0,k}/\ua.C_{0,k} & C_{0,k}\ar@{->>}[l] &
     T_k = \bigoplus\limits_{p+q=k} C_{p,q}\ar[l]\ar[d] 
\\  
                  &&C_{k,0} \ar@{->>}[d]
\\
                  &&X_k=C_{k,0}/\ux.C_{k,0}
\\
}
$$

\begin {fact} [Les complexes $X_\sbullet$ et $Y_\sbullet$]
\label {XYcomplexes}
On dispose d'isomorphismes canoniques:  
$$
X_\sbullet \simeq \rmK^{n-\sbullet}\big(\ua\,;\bbA/\idealx\big)
\qquad\qquad
Y_\sbullet \simeq \rmK^{n-\sbullet}\big(\ux \,;\bbA/\ideala\big)
$$
\end {fact}

De manière un peu abusive i.e. en négligeant les identifications,
on peut ainsi visualiser le $X$-complexe, qui, à l'indexation
$n-\sbullet$ près, n'est autre que le
complexe de Koszul montant de la suite $\ua$ sur $\bbA/\idealx$
$$
\newcommand \Aquoa {\bbA/\ideala}
\newcommand \Aquox {\bbA/\idealx}
\newcommand \egal{\rotatebox{90}{=}}
\xymatrix @R = 5pt @C=0.5cm{
0 &X_0\ar[l] &X_1 \ar[l]
 &\quad\cdots\ar[l]  &X_{n-1}\ar[l]  &X_n\ar[l] &0\ar[l]
\\
& \egal &\egal &\quad\cdots  &\egal  &\egal 
\\
0 &\rmK^n(\ua;\Aquox)\ar[l]  &\rmK^{n-1}(\ua;\Aquox) \ar[l]
 &\quad\cdots\ar[l]  &\rmK^1(\ua;\Aquox)\ar[l]  &\rmK^0(\ua;\Aquox)\ar[l] &0\ar[l]
}
$$
Quant au $Y$-complexe, c'est, modulo l'indexation $n-\sbullet$,
le complexe de Koszul montant de $\ux$ sur~$\bbA/\ideala$:
$$
\newcommand \Aquoa {\bbA/\ideala}
\newcommand \Aquox {\bbA/\idealx}
\newcommand \egal{\rotatebox{90}{=}}
\xymatrix @R = 0.1cm @C=0.5cm{
0 &\rmK^n(\ux;\Aquoa)\ar[l] &\rmK^{n-1}(\ux;\Aquoa)\ar[l] &\quad\cdots\ar[l]
  &\rmK^{1}(\ux;\Aquoa)\ar[l] &\rmK^0(\ux;\Aquoa)\ar[l] & 0\ar[l]
\\
& \egal &\egal &\quad\cdots  &\egal  &\egal 
\\
0 &Y_0\ar[l] &Y_1 \ar[l]
 &\quad\cdots\ar[l]  &Y_{n-1}\ar[l]  &Y_n\ar[l] &0\ar[l]
\\
}
$$
Lorsque $\ux$ est complètement sécante, le théorème du bicomplexe
fournit un morphisme de \emph {l'homologie} du $X$-complexe vers \emph
{l'homologie} du $Y$-complexe, qui est un isomorphisme si
$\ua$ est complètement sécante.

\subsubsection*{Au niveau homologique, en degré homologique $n$}

Examinons $\rmH(X_\sbullet)$ et $\rmH(Y_\sbullet)$ en degré
homologique $n$ (à droite sur les dessins ci-dessus)
\footnote{Le fait qui suit explique pourquoi le degré homologique~$n$
nous concerne tout particulièrement.}.
Les modules $\rmH_n(X_\sbullet) = \Ker(X_n \to X_{n-1})$
et $\rmH_n(Y_\sbullet) = \Ker(Y_n \to Y_{n-1})$ sont donc les
noyaux suivants:
$$
\Ker \Big(
\xymatrix@M=0.2pc{%
\rmK^0\big(\ua \,; \bbA/\idealx\big) \ar[r]^-{\left[
\begin{smallmatrix}
a_1 \\ \vdots \\ a_n \\
\end{smallmatrix}
\right]} 
& \rmK^1\big(\ua \,; \bbA/\idealx\big)
}
\Big) 
\qquad \text{et} \qquad 
\Ker \Big(
\xymatrix@M=0.2pc{%
\rmK^0\big(\ux \,; \bbA/\ideala\big) \ar[r]^-{\left[
\begin{smallmatrix}
x_1 \\ \vdots \\ x_n \\
\end{smallmatrix}
\right]} 
& \rmK^1\big(\ux \,; \bbA/\ideala\big)
}
\Big) 
$$
c'est-à-dire (en négligeant les identifications):
$$
\rmH_n(X_\sbullet) = \Ann\big(\ua \, ; \bbA/\idealx\big) 
\qquad \text{et} \qquad 
\rmH_n(Y_\sbullet) = \Ann\big(\ux \, ; \bbA/\ideala\big) 
$$
Le théorème du bicomplexe~\ref{BabyBicomplexTh} conduit donc au résultat suivant.

\begin{fact}[Ce que fournit le bicomplexe en degré homologique $n$]
\leavevmode
\label {XYhomologieDEGn}

Lorsque la suite $\ux$ est complètement sécante, on dispose d'une flèche 
$$
\xymatrix @M=0.4pc{
\Ann\big(\ua\,;\bbA/\idealx\big) \ar@{-->}[r]
& 
\Ann\big(\ux\,;\bbA/\ideala\big)}
$$
qui est un isomorphisme si $\ua$ est complètement sécante.
\end{fact}

\medskip

On notera l'utilisation d'une flèche pointillée sans nom qui
interviendra à plusieurs reprises, impossible de la rater: c'est celle
qui provient du bicomplexe!  De plus, ce pointillé est censé témoigner
d'un certain mystère dans l'obtention de cette flèche, mystère que
nous allons devoir lever en cas d'inclusion $\ideala \subset \idealx$
certifiée par une matrice de~$\bbM_n(\bbA)$ de déterminant $\nabla$.

\smallskip

Supposons $\ideala \subset\idealx$. Alors $\ua.(\bbA/\idealx) = 0$.
Puisque $X_\sbullet \simeq
\rmK^{n-\sbullet}\big(\ua\,;\bbA/\idealx\big)$ d'après le
fait~\ref{XYcomplexes}, nous allons pouvoir utiliser la remarque qui
suit.  Dans le contexte $\ideala\subset\idealx$ et $\ux$ complètement
sécante, on verra également comment déduire de cette remarque le fait
que $T_\sbullet \to X_\sbullet$ induit un isomorphisme en homologie
\emph{sans} utiliser le théorème du bicomplexe.  Nous aurons
cependant besoin du théorème du bicomplexe pour affirmer que
$T_\sbullet \to Y_\sbullet$ induit un isomorphisme en homologie
lorsque $\ua$ est complètement sécante.

\begin {fact}[Quand $\ua.E = 0$]
\label {aKillsE}
Soit $\ua$ une suite de scalaires vérifiant $\ua.E = 0$.  Alors le
complexe de Koszul montant $\rmK^\sbullet(\ua; E)$ est à
différentielles nulles.  En conséquence:
$$
\rmH^\sbullet(\ua;E) = \rmK^\sbullet(\ua;E)
$$
Si $\ua$ est de longueur~$n$:
$$
\rmH^k(\ua;E) = \BW^k(\bbA^n) \otimes_\bbA E \simeq E^{\binom{n}{k}}
$$
\end {fact}

Ainsi le complexe $X_\sbullet$ est à différentielles
nulles donc $\rmH(X_\sbullet) = X_\sbullet$.  En particulier,
$\rmH_n(X_\sbullet)$ est monogène engendré par $f_1 \wedge \cdots
\wedge f_n \bmod \idealx$ et $\rmH_n(X_\sbullet) \simeq
\Ann\big(\ua\,;\bbA/\idealx\big) = \bbA/\idealx$.

D'autre part, en raison de l'inclusion $\nabla
\idealx \subset \ideala$, on dispose d'une flèche $\bbA/\idealx \xrightarrow{\times
  \nabla} \bbA /\ideala$ dont l'image est incluse dans $\Ann(\ux \,;
\bbA /\ideala)$, donc d'un morphisme
$$
\bbA/\idealx \xrightarrow{\times\nabla} \Ann(\ux\,;\bbA/\ideala)
$$
Dans l'étude qui suit, figure le résultat suivant:

\begin {theo} [Identification du morphisme pour $\ideala\subset\idealx$ et
$\ux$ complètement sécante]\leavevmode
\label{PointilleIdentification}
  
Sous les hypothèses du titre, la flèche (en pointillé) donnée par le théorème du
bicomplexe s'identifie\footnote{Pour parvenir à la multiplication par
$\nabla$, il faut préciser les différentielles des divers
complexes qui interviennent. C'est ce que nous avons fait depuis le début!
Si avions obtenu $-\nabla$ au lieu de
$\nabla$, en ce qui concerne le théorème de Wiebe, cela serait sans
importance, n'est-ce-pas? Il y a des auteurs pas très sérieux que nous
ne nommerons pas. A contrario, on peut se fier par exemple à Bourbaki,
A.10, Algèbre homologique. On y voit par exemple dans la section \S4.1,
une définition très précise du produit tensoriel $C\otimes C'$ de deux
complexes (celle que nous avons adoptée)
et l'énoncé (proposition 2) suivant: $C\otimes C' \simeq C'\otimes C$
via $x \otimes x' \mapsto (-1)^{pq}\,x'\otimes x$ pour
$x \in C_p$, $x' \in C'_q$.} au morphisme ci-dessus noté~$\times\nabla$,
en rappellant que $\nabla = \det(V)$ où $V$ certifie
l'inclusion $\ideala\subset\idealx$.
\end {theo}

\subsubsection*{L'isomorphisme $T_\sbullet(\ua,\ux)\simeq\rmK^{2n-\sbullet}(\ua,\ux)$,
les suites $\ua,\ux$ étant quelconques de même longueur~$n$}

Nous allons utiliser l'isomorphisme de complexes
$\rmK^\sbullet(\ua)\otimes\rmK^\sbullet(\ux)\simeq\rmK^\sbullet(\ua,\ux)$
mentionné en \ref{EncadrePrecision} (page~\pageref{EncadrePrecision}).
Le second modèle, qui ne fait plus référence au produit tensoriel, est
le brave complexe de Koszul montant de la suite~$(\ua,\ux)$ de longueur
$2n$, complexe directement basé sur l'algèbre extérieure. Par
exemple, les différentielles (de degré $1$) sont les multiplications à gauche
$(a+x) \wedge\sbullet$ par le vecteur $a+x \in \rmK^1(\ua)\oplus\rmK^1(\ux)$.
Certains calculs sont parfois plus simples à réaliser dans ce second
modèle.

\medskip

Cet isomorphisme s'applique au complexe total
$T_\sbullet=T_\sbullet(\ua,\ux)$ de $C_{\sbullet,\sbullet}$.  Mais eu
égard au renversement initial $n-\sbullet$ dans la définition de
$C_{\sbullet,\sbullet}$, le complexe $T_\sbullet$ s'identifie au
complexe $\rmK^{2n-\sbullet}(\ua,\ux)$:
$$
T_\sbullet \overset{\rm can.}{\simeq} \rmK^{2n-\sbullet}(\ua,\ux)
$$
Il est indispensable de maitriser l'isomorphisme et
de ne pas trop privilégier un modèle à l'autre
car certaines primitives s'expriment plus facilement d'un côté
que de l'autre. Par exemple les projections de~$T_\sbullet$
sur le $X$-complexe ou le $Y$-complexe sont aisées
à définir sur le modèle initial par l'intermédiaire
des projections de $T_k=\bigoplus\limits_{p+q=k} C_{p,q}$
sur $C_{k,0}$ ou $C_{0,k}$.

\bigskip

Prenons l'exemple du degré homologique $n$ qui nous concerne tout particulièrement.
$$
T_n \overset{\rm def.}{=} \bigoplus_{p+q=n} C_{p,q}
\overset{\rm def.}{=} \bigoplus_{p+q=n} \rmK^{n-p}(\ua)\otimes \rmK^{n-q}(\ux)
\quad \overset{\rm can.}{\simeq}\quad
\rmK^n(\ua,\ux) = \BW^n\big(\rmK^1(\ua)\oplus \rmK^1(\ux)\big) 
$$
La séparation des $(p,q)$ en $(0,n)$, $(n,0)$ et autres
fournit une décomposition $T_n = C_{0,n}\oplus S\oplus C_{n,0}$. Il lui
correspond dans $\rmK^n(\ua,\ux)$ la somme directe des 3 sous-modules libres:
$$
\rmK^n(\ua) = \BW^n\rmK^1(\ua)
\qquad\quad
S' = \bigoplus_{p+q=n\atop p,q\ne 0} \BW^p\rmK^1(\ua) \ \wedge\ \BW^q\rmK^1(\ux)
\qquad\quad
\rmK^n(\ux) = \BW^n\rmK^1(\ux)
$$
Si nous notons $(e_1,\dots,e_n)$ la base canonique de $\rmK^1(\ua)$ et
$(f_1,\dots,f_n)$ celle de $\rmK^1(\ux)$:
$$
\rmK^n(\ua) = \bbA\,e_1\wedge\cdots\wedge e_n,
\qquad\quad
S' = \bigoplus_{\#I+\#J=n\atop I,J\ne\emptyset} \bbA\,e_I \wedge f_J,
\qquad\quad
\rmK^n(\ux) = \bbA\,f_1\wedge\cdots\wedge f_n
$$
Les deux projections de $T_n$ sur $C_{0,n}$ et sur $C_{n,0}$ s'identifient au 
deux projections:
$$
\xymatrix @C=0.4cm {
\BW^n\big(\rmK^1(\ua)\oplus\rmK^1(\ux)\big)
\ar@/^15pt/[r]^{p_\ua} \ar@/_15pt/[rrr]_{p_\ux}
\quad = &
\BW^n\rmK^1(\ua) &\oplus\quad S'\quad \oplus &
\BW^n\rmK^1(\ux)
}
$$
Avec un peu d'habitude, nous naviguerons librement entre le modèle
initial de $T_\sbullet$ et le modèle $\rmK^{2n-\sbullet}(\ua,\ux)$.
En guise de résumé, voici les projections $p_\ua : T_n \to C_{0,n}$
et $p_\ux : T_n \to C_{n,0}$ dans le second modèle de~$T_n$:
$$
C_{0,n} \overset{\rm def.}{=} \rmK^n(\ua) = \BW^n\rmK^1(\ua)
\qquad
C_{n,0} \overset{\rm def.}{=} \rmK^n(\ux) = \BW^n\rmK^1(\ux)
\qquad
T_n = \BW^n\big(\rmK^1(\ua)\oplus\rmK^1(\ux)\big)
$$
Pour $r_1, \dots, r_n \in \rmK^1(\ua)$ et $s_1, \dots, s_n \in \rmK^1(\ux)$, l'élément
$\bfu = (r_1+s_1) \wedge \cdots\wedge(r_n+s_n)$ de $T_n$ a pour projections:
$$
p_\ua(\bfu) = r_1 \wedge\cdots\wedge r_n, \qquad\qquad
p_\ux(\bfu) = s_1 \wedge\cdots\wedge s_n
$$

\subsubsection*{Stratégies pour la preuve du théorème~\ref{PointilleIdentification}
et du théorème de Wiebe}

Pour venir à bout de ce théorème, nous allons étudier deux approches
présentant un certain nombre de points communs.

\begin {enumerate}
\item
Une approche relativement concrète et précise, qui fait l'objet de la proposition suivante. 

\item
Une approche plus abstraite du type ``reculer pour mieux sauter'' qui
repose sur l'aspect fonctoriel d'une certaine transformation. Cette
approche, en certain sens plus générale, fait intervenir trois suites
$\ua', \ua, \ux$ de même longueur $n$, la \emph {seule} hypothèse
étant $\langle\ua\rangle \subset\langle\ua'\rangle$ certifiée par une
matrice $V \in \bbM_n(\bbA)$ transformant $\ua'$ en $\ua$.  La
transformation fonctorielle en question, portée par $(\ua',V)$, est
un morphisme de bicomplexes:
$$
C_{\sbullet,\sbullet}(V) :
C_{\sbullet,\sbullet}(\ua',\ux) \to  C_{\sbullet,\sbullet}(\ua,\ux)
\qquad\qquad
\text{où $\ua := \ua'. \transpose{V}$}
$$

Elle induit des morphismes au niveau des complexes totaux et des
$X$-complexes et $Y$-complexes que nous étudierons en degré
homologique~$n$.  Ce n'est qu'à la fin de l'étude que nous récolterons
une \fbox {version affaiblie} du théorème~\ref{PointilleIdentification}
en prenant $\ua' = \ux$ et en supposant $\ux$ complètement sécante.
Cette version affaiblie nous suffira à montrer le théorème de Wiebe.
\end {enumerate}

\medskip

Voici l'énoncé du résultat concret qui précise le théorème~\ref{PointilleIdentification}.
La preuve est reportée en page~\pageref{ConcreteSolutionProof}.

\begin {prop} [Approche dite concrète]
\leavevmode
\label {ConcreteSolution}

Le contexte est le suivant: une inclusion $\ideala\subset\idealx$
certifiée par $V = (v_{ij})_{i,j} \in \bbM_n(\bbA)$ vérifiant:
$$
\begin{bmatrix}a_1\\ \vdots\\ a_n\\ \end{bmatrix} =
V\begin {bmatrix} x_1\\ \vdots\\ x_n\end{bmatrix}
$$
On note $\bfe = (e_1, \dots, e_n)$ la base canonique de $\rmK^1(\ua)$
et $\bff = (f_1, \dots, f_n)$ celle de $\rmK^1(\ux)$. En posant
$$
a = \sum_i a_ie_i \in \rmK^1(\ua), \qquad\qquad  x = \sum_j x_jf_j \in \rmK^1(\ux)
$$
on voit $V$ comme la matrice d'une application linéaire $v : \big(\rmK^1(\ux),\bff)
\to \big(\rmK^1(\ua),\bfe\big)$ vérifiant $v(x) = a$.

\begin {enumerate} [\rm i)]
\item    
Il existe $\bfw \in T_n $, cycle de $T_\sbullet$, dont les projections
sur $C_{0,n}$ et $C_{n,0}$ sont les suivantes:
$$
\xymatrix {
C_{0,n} &T_n \ar[l]\ar[d]\\
       &C_{n,0}\\
}
\qquad\qquad
\xymatrix @C = 0.2cm{
\det(V) \times e_1\wedge\cdots\wedge e_n&\bfw\ar[l]\ar[d] \\
                                          &f_1\wedge\cdots\wedge f_n \\
}                                               
$$
Le ``Il existe'' est tout ce qu'il y a de plus explicite puisque dans
le modèle $\rmK^{2n-\sbullet}(\ua,\ux)$ de $T_\sbullet$:
$$
\bfw = (v(f_1)+f_1)\wedge \cdots \wedge (v(f_n)+f_n) \in \BW^n\big(\rmK^1(\ua)\oplus\rmK^1(\ux)\big)
$$
On suppose de surcroît \fbox{$\ux$ complètement sécante}.

\item
La projection $T_\sbullet \to X_\sbullet$ induit\footnote{Insistons encore
sur le fait que la preuve fournie
exploite de manière essentielle l'inclusion $\ideala\subset\idealx$
et n'utilise pas le théorème du bicomplexe.
Cependant, nous solliciterons ce dernier
pour justifier le fait que $T_\sbullet \to Y_\sbullet$
induit un isomorphisme en homologie lorsque $\ua$ est
complètement sécante.} un isomorphisme en homologie.  

\item
La flèche en pointillé transforme le générateur 
$f_1\wedge\cdots\wedge f_n \bmod\idealx$ de $\rmH_n(X_\sbullet)
\simeq \bbA/\idealx$ en
l'élément $\det(V)\times e_1\wedge\cdots\wedge e_n\bmod\ideala$
de $\rmH_n(Y_\sbullet) \simeq \Ann(\ux;\bbA/\ideala)$:
$$
\xymatrix {
\rmH_n(Y_\sbullet)  &\rmH_n(T_\sbullet) \ar[l]\ar[d]\\
                 &\rmH_n(X_\sbullet)\ar@{-->}[ul] \\
}
\qquad\qquad
\xymatrix {
\det(V)\times e_1\wedge\cdots\wedge e_n\bmod\ideala   &\bfw\ar[l]\ar[d] \\
                                           &f_1\wedge\cdots\wedge f_n \bmod\idealx \ar@{-->}[ul]
}                                               
$$
\end {enumerate}
\end {prop}

\subsubsection*{L'homologie en tout degré des complexes $X_\sbullet$ et $Y_\sbullet$
($\ua$, $\ux$ quelconques de même longueur~$n$)}

Malgré le fait que nous soyons essentiellement intéressés par les
complexes~$X_\sbullet$ et $Y_\sbullet$ en degré homologique~$n$, il
n'est pas difficile, puisque $X_\sbullet \simeq
\rmK^{n-\sbullet}\big(\ua\,;\bbA/\idealx\big)$ et $Y_\sbullet \simeq
\rmK^{n-\sbullet}\big(\ux \,;\bbA/\ideala\big)$, de préciser, sans
aucune hypothèse sur $\ua$ et $\ux$, le statut de leur homologie en
tout degré homologique~$k$.  Nous l'utiliserons pour montrer, lorsque
$\ideala \subset \idealx$ et $\ux$ est complètement sécante, que
$T_\sbullet \to X_\sbullet$ induit un isomorphisme en homologie en
nous passant du théorème du bicomplexe.

\begin{fact}[L'homologie du $X$-complexe et $Y$-complexe en tout degré homologique]
\leavevmode
\label {XYhomologie}

En tout degré homologique $k$, on dispose d'isomorphismes canoniques:
$$
\rmH_k(Y_\sbullet) \simeq \rmH^{n-k}(\ux;\bbA/\ideala),
\qquad
\rmH_k(X_\sbullet) \simeq \rmH^{n-k}(\ua;\bbA/\idealx)
$$
Puisque de manière générale, $\rmH^0(\ua;E) = \Ann(\ua;E)$, on
retrouve dans le cas $k=n$ pour lequel $n-k=0$, les égalités/isomorphies qui
précédent le fait \ref{XYhomologieDEGn}.
\end{fact}

Nous aurons besoin du lemme classique suivant relatif au complexe 
de Koszul montant.

\begin {lem} \leavevmode
\label {uaSUBux}
Soient deux suites de scalaires $\ua = (a_1, \dots, a_m)$ et
$\ux = (x_1, \dots, x_n)$ telles que $\ideala \subset \idealx$.
On dispose alors d'un isomorphisme canonique de complexes:
$$
\rmK^\sbullet(\ua,\ux) \simeq \BW^\sbullet(\bbA^m) \otimes \rmK^\sbullet(\ux)
$$
où $\BW^\sbullet(\bbA^m)$ est considéré comme un complexe à différentielles
nulles. En particulier, pour les termes des complexes:
$$
\rmK^k(\ua,\ux) \simeq \bigoplus_{i+j=k} \BW^i(\bbA^m) \otimes \rmK^j(\ux)
$$
Cet isomorphisme de complexes induit un isomorphisme de
$\bfA$-modules gradués:
$$
\rmH^\sbullet(\ua,\ux) \simeq \BW^\sbullet(\bbA^m) \otimes \rmH^\sbullet(\ux)
\qquad \text{qui se traduit en} \qquad
\rmH^k(\ua,\ux) \simeq \bigoplus_{i+j=k} \BW^i(\bbA^m) \otimes \rmH^j(\ux)
$$
\end {lem}  

On pourra en trouver une preuve dans Bruns-Herzog \cite [section I.6,
  proposition I.6.21]{BrunsHerzog} ou la thèse de C.~Tête en
\cite[chap. I, propositions 4.1 et 4.2]{Tete}.  Nous en indiquons le
principe.

\begin {proof} [Preuve (sketch)]

Elle consiste à remplacer la suite $\ua$ par la suite $\uzero_m$ constituée de $m$ zéros.
On introduit $a \in \bbA^m$, $x \in \bbA^n$,   
$W \in \bbM_{m,n}(\bbA)$ traduisant l'inclusion $\ideala\subset\idealx$
et $W' \in \bbM_{m+n}(\bbA)$:  
$$
a = \begin {bmatrix} a_1\\ \vdots\\ a_m \end{bmatrix},
\qquad
x = \begin {bmatrix} x_1\\ \vdots\\ x_n \end{bmatrix},
\qquad
a = W x,
\qquad
W' = \begin {bmatrix} \Id_m & W\\ 0_{n\times m} & \Id_n \end {bmatrix},
\qquad
W' \begin {bmatrix} 0_m\\ x \end {bmatrix} =  \begin {bmatrix} a\\ x \end {bmatrix}
$$
Cette matrice inversible $W'$ induit un isomorphisme de complexes
$\rmK^\sbullet(\uzero_m,\ux) \simeq \rmK^\sbullet(\ua,\ux)$. On termine
en utilisant les isomorphismes de complexes:
$$
\rmK^\sbullet(\uzero_m,\ux) \simeq \rmK^\sbullet(\uzero_m) \otimes \rmK^\sbullet(\ux)
\qquad\qquad
\rmK^\sbullet(\uzero_m) \simeq \BW^\sbullet(\bbA^m)
$$
\end {proof}

\begin {proof} [Preuve de la proposition \ref{ConcreteSolution}]  
\leavevmode
\label{ConcreteSolutionProof}

\noindent
Point i). Rappelons que
$$
C_{n,0} \overset{\rm def.}{=} \rmK^n(\ux) = \bbA.f_1 \wedge\cdots\wedge f_n,
\qquad
C_{0,n} \overset{\rm def.}{=} \rmK^n(\ua) = \bbA.e_1 \wedge\cdots\wedge e_n,
\qquad
T_n = \BW^n\big(\rmK^1(\ua)\oplus\rmK^1(\ux)\big)
$$
Notons $p_\ux : T_n \to C_{n,0}$ et $p_\ua : T_n \to C_{0,n}$ les projections.
Pour $r_1, \dots, r_n \in \rmK^1(\ua)$ et $s_1, \dots, s_n \in \rmK^1(\ux)$, l'élément
$\bfu = (r_1+s_1) \wedge \cdots\wedge(r_n+s_n)$ de $T_n$ a pour projections:
$$
p_\ua(\bfu) = r_1 \wedge\cdots\wedge r_n, \qquad\qquad
p_\ux(\bfu) = s_1 \wedge\cdots\wedge s_n
$$
En posant $\bfw = (v(f_1)+f_1)\wedge \cdots \wedge (v(f_n)+f_n)$, on a donc:
$$
p_\ua(\bfw) = v(f_1) \wedge\cdots\wedge v(f_n) = \det(V)\, e_1 \wedge\cdots\wedge e_n,
\qquad\qquad
p_\ux(\bfw) = f_1 \wedge\cdots\wedge f_n
$$
Reste à montrer que $\bfw$ est un cycle de $T_\sbullet = \rmK^\sbullet(\ua,\ux)$
i.e., par définition du complexe de Koszul montant
de la suite $(\ua\,\ux)$, que $(a+x) \wedge \bfw = 0$. Introduisons
l'endomorphisme~$v'$ de $\rmK^1(\ua) \oplus\rmK^1(\ux)$ défini par:
$$
v' =
\NorthEastBordermatrix {
\bfe   & \bff  &  \\
\Id_n  &  V    & \bfe \\
0      & \Id_n & \bff \\
}
$$
On a en particulier $v'(f_j) = v(f_j) + f_j$ de sorte que $v'(x) =
v(x) + x = a + x$. On peut exprimer $\bfw$ à l'aide de $v'$:
$$
\bfw = v'(f_1) \wedge \cdots \wedge v'(f_n)
$$
Puisque $v'(x) = a+x$, on doit démontrer que $v'(x) \wedge \bfw = 0$ c'est-à-dire:
$$
v'(x) \wedge v'(f_1) \wedge \cdots \wedge v'(f_n) = 0
\qquad \text{ou encore} \qquad
\BW^{n+1}(v')\,(x\wedge f_1\wedge\cdots\wedge f_n) = 0
$$
Mais $x$ est une combinaison linéaire des $f_j$ donc $x\wedge f_1\wedge\cdots\wedge f_n= 0$,
a fortiori l'égalité ci-dessus.

\bigskip
\noindent
Point ii)
Il s'agit de montrer que $\rmH(T_\sbullet) \to \rmH(X_\sbullet)$ est un isomorphisme.
Nous allons démontrer que, canoniquement\footnote{Un adverbe très en vogue dans cette
annexe.}, on a $\rmH_k(T_\sbullet) \simeq \rmH_k(X_\sbullet)$, 
en nous dispensant de vérifier que cet isomorphisme canonique est celui
induit par la projection.

$\blacktriangleright$
On s'occupe d'abord de $\rmH(T_\sbullet)$.
Puisque $\ideala\subset\idealx$, on dispose d'après le lemme~\ref{uaSUBux}
d'un isomorphisme de complexes:
$$
T_\sbullet = \rmK^{2n-\sbullet}(\ua,\ux) \simeq
\BW^{n-\sbullet}(\bbA^n) \otimes \rmK^{n-\sbullet}(\ux)
$$
où $\BW^{n-\sbullet}(\bbA^n)$ est considéré comme un complexe à différentielles
nulles. En homologie, il vient
$$
\rmH(T_\sbullet) = \rmH^{2n-\sbullet}(\ua,\ux) \simeq
\BW^{n-\sbullet}(\bbA^n) \otimes \rmH^{n-\sbullet}(\ux)
$$
En conséquence:
$$
\rmH_k(T_\sbullet) = \rmH^{2n-k}(\ua,\ux) \simeq
\bigoplus_{p+q=k}\BW^{n-p}(\bbA^n) \otimes \rmH^{n-q}(\ux)
$$
On utilise maintenant l'information $\ux$ complètement sécante.
On a donc $\rmH^{n-q}(\ux)=0$ pour $q\ne 0$ et
$\rmH^{n}(\ux) \simeq \bbA/\idealx$. Ci-dessus, il reste donc
le terme en $(p=k,q=0)$:
$$
\rmH_k(T_\sbullet) \simeq \BW^{n-k}(\bbA^n) \otimes \rmH^{n}(\ux) \simeq
\BW^{n-k}(\bbA^n) \otimes \bbA/\idealx
\leqno(\star)
$$

$\blacktriangleright$
On s'occupe de $\rmH(X_\sbullet)$. D'après le fait~\ref{XYcomplexes},
on a $X_\sbullet \simeq \rmK^{n-\sbullet}(\ua;\bbA/\idealx)$. Mais
$\ua.(\bbA/\idealx) = 0$ puisque $\ideala\subset\idealx$. D'après le
fait~\ref{aKillsE}, le complexe $X_\sbullet$ est à différentielles nulles donc
$$
\rmH_k(X_\sbullet) = \rmK_k(X_\sbullet) \simeq \rmK^{n-k}(\ua;\bbA/\idealx)
= \BW^{n-k}(\bbA^n) \otimes \bbA/\idealx
\leqno(\star\star)
$$
Il ne reste plus au lecteur qu'à constater que $(\star)$ et $(\star\star)$ coïncident.

\bigskip
\noindent
Point iii)
Découle des deux points précédents.
\end {proof}


\subsection {Transformations fonctorielles attachées à
$\rmK^\sbullet(\protect\ub,\protect\uc)\simeq\rmK^\sbullet(\protect\ub)\otimes\rmK^\sbullet(\protect\uc)$}

Nous rappelons que cet isomorphisme, noté $(\heartsuit)$ dans la
section~\ref{EncadrePrecision}, est un isomorphisme canonique de
complexes différentiels (ou d'algèbres graduées alternées) et que
$\ub, \uc$ sont deux suites de scalaires de longueurs quelconques).
Nous fournissons ici sans justification un certain nombre de transformations
fonctorielles attachées à cet isomorphisme.  Nous utilisons les
notations ci-dessous qui nous semblent adaptées au traitement en
cours, bien que plus générales. Dans la section suivante, nous
utiliserons ces transformations dans le cas d'égalités $n=n'=m=m'$;
elles seront parfois un tantinet obscurcies à cause du désagrément
provoqué par le renversement $n-\sbullet$.

\medskip

Soient $\ua$ (resp. $\ua'$) une suite de longueur~$m$ (resp. $m'$)
et $\ux$ (resp. $\ux'$) une suite de longueur~$n$ (resp.~$n'$).
Les isomorphismes canoniques sur la sellette sont:
$$
\rmK^\sbullet(\ua',\ux') \simeq \rmK^\sbullet(\ua')\otimes\rmK^\sbullet(\ux'),
\qquad\qquad
\rmK^\sbullet(\ua,\ux) \simeq \rmK^\sbullet(\ua)\otimes\rmK^\sbullet(\ux)
$$
Soient des matrices $V \in \bbM_{m,m'}(\bbA)$, $W \in \bbM_{n,n'}(\bbA)$ telles que
$$
\begin{bmatrix} a_1 \\ \vdots \\ a_m \end{bmatrix} =
V \begin{bmatrix} a'_1 \\ \vdots \\ a'_{m'} \end{bmatrix}
\qquad\qquad
\begin{bmatrix} x_1 \\ \vdots \\ x_n \end{bmatrix} =
W \begin{bmatrix} x'_1 \\ \vdots \\ x'_{n'} \end{bmatrix}
$$
On dispose ainsi d'un certain nombre de vecteurs dans les
divers $\rmK^1(\quad)$:
$$
a' \in \rmK^1(\ua'),\  a \in \rmK^1(\ua), 
\qquad
x' \in \rmK^1(\ux'),\  x \in \rmK^1(\ux),
\qquad
x' + a' \in \rmK^1(\ux',\ua') = \rmK^1(\ux')\oplus\rmK^1(\ua')
\quad \text{etc.}
$$
et d'applications linéaires:
$$
v : \rmK^1(\ua') \to \rmK^1(\ua), \ v(a') = a,
\qquad
w : \rmK^1(\ux') \to \rmK^1(\ux), \ w(x') = x,
\qquad
v \oplus w  : \rmK^1(\ua',\ux') \to \rmK^1(\ua,\ux)
$$
L'application linéaire $v$ se prolonge en un morphisme de complexes
$\BW^\sbullet : \rmK^\sbullet(\ua') \to \rmK^\sbullet(\ua)$. Idem
pour les autres. On dispose alors d'une identification
$$
\BW^\sbullet(v) \otimes \BW^\sbullet(w) \simeq \BW^\sbullet(v \oplus w) 
$$
Ce qui signifie que le diagramme suivant est commutatif:
$$
\xymatrix @C = 3.5cm {
\rmK^r(\ua',\ux')\ar[d]|\simeq\ar[r]^-{\BW^r(v \oplus w)}
      &\rmK^r(\ua,\ux) \ar[d]|\simeq
\\
\big(\rmK^\sbullet(\ua')\otimes\rmK^\sbullet(\ux')\big)_r
\ar[r]_{\big(\BW^\sbullet(v) \otimes \BW^\sbullet(w)\big)_r} &
\big(\rmK^\sbullet(\ua)\otimes\rmK^\sbullet(\ux)\big)_r
\\
}
$$
où la flèche horizontale basse est bien entendu, en posant
$u_{p,q} = \BW^p(v) \otimes \BW^q(w)$:
$$
\big(\BW^\sbullet(v) \otimes \BW^\sbullet(w)\big)_r =
\bigoplus_{p+q=r} u_{p,q} :
\xymatrix {
\displaystyle \bigoplus_{p+q=r}\rmK^p(\ua')\otimes\rmK^q(\ux') \ar[r] &
\displaystyle \bigoplus_{p+q=r}\rmK^p(\ua)\otimes\rmK^q(\ux)
}
$$
Nous utiliserons le cas particulier où $\ux' = \ux$ et $w$ est
l'identité que nous écrivons $w = \Id_\ux$.
Le diagramme précédent s'écrit alors:
$$
\xymatrix @C = 3.5cm {
\rmK^r(\ua',\ux)\ar[d]|\simeq\ar[r]^-{\BW^r(v \oplus \Id_\ux)}
       &\rmK^r(\ua,\ux) \ar[d]|\simeq
\\
\big(\rmK^\sbullet(\ua')\otimes\rmK^\sbullet(\ux)\big)_r
\ar[r]_{\big(\BW^\sbullet(v) \otimes \BW^\sbullet(\Id_\ux)\big)_r} &
\big(\rmK^\sbullet(\ua)\otimes\rmK^\sbullet(\ux)\big)_r
\\
}
$$

\subsection {La transformation $C_{\sbullet,\sbullet}(\protect\ua',\protect\ux) \to
C_{\sbullet,\sbullet}(\protect\ua,\protect\ux)$ associée à $V \in \bbM_n(\bbA)$
où $\protect\ua=\protect\ua'.\transpose{V}$}

Le contexte de cette section est le suivant: trois suites de scalaires de même longueur $n$
$$
\ua' = (a'_1, \dots, a'_n), \quad
\ua = (a_1, \dots, a_n), \quad
\ux = (x_1, \dots, x_n), \quad
\text{avec la relation d'inclusion $\ideala \subset \idealap$}
$$
Pour des raisons structurelles, nous introduisons les vecteurs $a$ et $a'$ et
nous certifions l'inclusion $\ideala \subset \idealap$ en colonnes via
une matrice $V \in \bbM_n(\bbA)$:
$$
a = \begin{bmatrix} a_1 \\ \vdots \\ a_n  \end{bmatrix} \in \rmK^1(\ua),
\qquad\quad
a' = \begin{bmatrix} a'_1 \\ \vdots \\ a'_n  \end{bmatrix} \in \rmK^1(\ua'),
\qquad\quad
a = V.a'
$$
Cela nous permet de voir $V$ comme une application linéaire $v :
\rmK^1(\ua') \to \rmK^1(\ua)$ vérifiant $v(a') = a$.  Nous constatons
ainsi que $a$ est porté par $(a',v)$ puisque $a = v(a')$.  Nous allons
voir que ce sont les propriétés fonctorielles induites par $v$ qui
vont nous permettre de montrer une \fbox {version affaiblie} du
théorème~\ref{PointilleIdentification} en enrichisant le
résultat~\ref{XYhomologieDEGn} de la manière suivante:

\begin {theo}
\label {Vfonctorialite}
Lorsque la suite $\ux$ est complètement sécante, le théorème du bicomplexe 
fournit un diagramme commutatif en cohomologie Koszul :
$$
\xymatrix@M=0.5pc @C=4pc{  
\Ann\big(\ua' \, ; \bbA/\idealx \big)\ar@[blue][d]_-{\times 1}   \ar@{-->}[r] 
&  \Ann\big(\ux \, ; \bbA/\idealap \big)\ar@[blue][d]^-{\times \det(V)} \\
\Ann\big(\ua \, ; \bbA/\idealx \big)   \ar@{-->}[r] & \Ann\big(\ux \, ; \bbA/\ideala \big) 
}
$$
Si de plus $\ua'$ (resp. $\ua$) est complètement sécante, la flèche horizontale pointillée
haute (resp. basse) est un isomorphisme.
\end {theo}

En apparence, nous avons doublé le mystère des flèches en pointillé:
l'horizontale basse (initiale) et et l'horizontale haute de même
nature. Mais elles sont reliées par la commutativité du diagramme et
par la connaissance des flèches verticales, ce qui change tout!  Nous
en tirons tout de suite le bénéfice et reportons la preuve de ce théorème
à la suite.

\subsubsection*{Preuve d'une version affaiblie du théorème \ref{PointilleIdentification} et
preuve du théorème de Wiebe.}

$\blacktriangleright$
Reprenons le contexte du théorème \ref{PointilleIdentification},
c'est-à-dire une inclusion $\ideala \subset \idealx$ certifiée par une
matrice $V \in \bbM_n(\bbA)$ avec $\ux$ complètement sécante.  Faisons
$\ua' = \ux$ dans le diagramme du théorème~\ref{Vfonctorialite} si
bien que \emph {l'horizontale haute} dans le diagramme ci-dessous est
un isomorphisme.
$$
\xymatrix @M=0.4pc @C=5pc{
\bbA /\idealx  \ar@{-->}[r]^-{\simeq}
\ar@<-2pt>[d]_{\times 1} 
& 
\bbA /\idealx
\ar@<-2pt>[d]^-{\times\det(V)} 
\\
\bbA/\idealx \ar@{-->}[r]
& 
\Ann\big(\ux \, ; \bbA/\ideala\big) 
}
$$
La commutativité du diagramme donne comme information le fait que la flèche pointillée
basse est $\times\det(V)$ à \fbox {multiplication près par un inversible}
de l'anneau $\bbA/\idealx$ de départ.

\medskip 
$\blacktriangleright$
Revenons au contexte du théorème de Wiebe, c'est-à-dire à une inclusion 
$\ideala \subset \idealx$ certifiée une matrice de déterminant~$\nabla$  avec 
$\ua$ complètement sécante.
Faisons $\ua' = \ux$ dans le diagramme du théorème~\ref{Vfonctorialite}.
Comme $\ua$ est supposée complètement sécante, la suite~$\ux$ l'est également 
si bien que les horizontales dans le diagramme ci-dessous sont des isomorphismes.
$$
\xymatrix @M=0.4pc @C=5pc{
\bbA /\idealx  \ar@{-->}[r]^-{\simeq}
\ar@<-2pt>[d]_{\times 1}^-{\simeq} 
& 
\bbA /\idealx
\ar@<-2pt>[d]^-{\times \nabla} 
\\
\bbA/\idealx \ar@{-->}[r]^-{\simeq}
& 
\Ann\big(\ux \, ; \bbA/\ideala\big) 
}
$$
Il ne reste plus qu'à récolter : la flèche verticale de droite 
$\bbA/\idealx \xrightarrow{\times \nabla} \Ann\big(\ux \,; \bbA/\ideala)$ est un isomorphisme. 
C'est exactement la reformulation du théorème de Wiebe.

\subsubsection*{Ce que l'application linéaire $v$ induit au niveau des
                complexes de Koszul de $\ua'$ et $\ua$}

Elle se prolonge à l'algèbre extérieure et, grâce à l'égalité $a' =
v(a)$, fournit le morphisme de complexes $\bigwedge^\sbullet v :
\rmK^\sbullet(\ua') \rightarrow \rmK^\sbullet(\ua)$
$$
\xymatrix @M=0.5pc @C=3pc{
\bbA/\idealap \ar[d]^-{\times\det(V)} & \rmK^n(\ua') \ar@{->>}[l] \ar[d]^-{\bigwedge^n v}
& \qquad \cdots \qquad \ar[l] & \rmK^1(\ua') \ar[l] \ar[d]^-{v} & \rmK^0(\ua') \ar[l] \ar[d]^-{\Id}\\
\bbA/\ideala & \rmK^n(\ua) \ar@{->>}[l] 
& \qquad \cdots \qquad \ar[l] & \rmK^1(\ua) \ar[l] & \rmK^0(\ua) \ar[l] \\
}
$$

\subsubsection*{Ce que l'application linéaire $v$ induit au niveau des bicomplexes $C_{\sbullet,\sbullet}$}

Nous définissons les bicomplexes $C'_{\sbullet, \sbullet} =
\rmK^{n-\sbullet}(\ua') \otimes \rmK^{n-\sbullet}(\ux)$ et
$C_{\sbullet, \sbullet} = \rmK^{n-\sbullet}(\ua) \otimes
\rmK^{n-\sbullet}(\ux)$ et nous identifions leurs complexes totaux
respectivement à $T'_\sbullet = \rmK^{2n-\sbullet}(\ua', \ux)$ et
$T_\sbullet = \rmK^{2n-\sbullet}(\ua, \ux)$.

\smallskip

En tenant compte du renversement $n-\sbullet$ (c'est un peu pénible),
le morphisme de complexes $\bigwedge^\sbullet v$ permet de définir un
morphisme au niveau des bicomplexes, morphisme que nous notons
$C_{\sbullet,\sbullet}(v)$:
$$
C_{\sbullet,\sbullet}(v) := \BW^{n-\sbullet}(v) \otimes \Id :
\xymatrix{
C'_{\sbullet, \sbullet} \ar@[blue][r] & C_{\sbullet, \sbullet}
} 
$$
Nous allons cerner les diverses projections de sources les complexes totaux. Nous introduisons:
$$
v\oplus\Id_\ux : \rmK^1(\ua')\oplus\rmK^1(\ux) \longrightarrow \rmK^1(\ua)\oplus\rmK^1(\ux)
$$
Et nous notons $T_k(v) : T'_k \to T_k$ le morphisme induit par $v$ au niveau
des complexes totaux:
$$
T'_k = \rmK^{2n-k}(\ua',\ux),\qquad   T_k = \rmK^{2n-k}(\ua,\ux), \qquad
T_k(v) : \xymatrix @C=2.5cm {T'_k \ar[r]^{\BW^{2n-k}(v\oplus\Id_\ux)} & T_k}
$$
Bien entendu, il y a eu identification
$$
\bigoplus_{p+q=k} \BW^{n-p}(v) \otimes \BW^{n-q}(\Id_\ux)
\qquad \longleftrightarrow \qquad
\BW^{2n-k}(v\oplus\Id_\ux)
$$
Nous obtenons ainsi, en préparation des projections sur le
$Y$-complexe (à gauche) et sur le $X$-complexe (à droite), des diagrammes
commutatifs:
$$
\xymatrix @R=1cm @C=3cm{
T'_k \ar[d]|{\text{projection}}\ar[r]^{T_k(v)} &T_k\ar[d]|{\text{projection}}
\\  
C'_{0,k} \ar[r]_{\BW^{n}(v)\otimes\BW^{n-k}(\Id_\ux)}       &C_{0,k}
}
\qquad\qquad
\xymatrix @R=1cm @C=3cm{
T'_k \ar[d]|{\text{projection}}\ar[r]^{T_k(v)} &T_k\ar[d]|{\text{projection}}
\\  
C'_{k,0} \ar[r]_{\BW^{n-k}(v)\otimes\BW^n(\Id_\ux)}       &C_{k,0}
}
$$
Il est instructif d'y faire $k=n$, c'est le cas qui nous concerne tout particulièrement.
A gauche, l'horizontale basse \og sent\fg{} le déterminant $\det(V)$ puisqu'elle est définie par
$$
C'_{0,n}=\rmK^n(\ua'),\qquad C_{0,n}=\rmK^n(\ua),\qquad\qquad
\BW^n(v) : \rmK^n(\ua') \to \rmK^n(\ua)
$$
Tandis qu'à droite, il s'agit de l'identité qui \og sent\fg{} le 1:
$$
C'_{n,0} = C_{n,0} = \rmK^n(\ux), \qquad\qquad
\BW^n(\Id_\ux) : \rmK^n(\ux) \to \rmK^n(\ux)
$$
Revenons en degré homologique quelconque $k$.
On a vu (cf le fait~\ref{XYcomplexes} et avant)
que l'on passait de $C_{0,k}$ à $Y_k$ en divisant par l'idéal $\ideala$
et de $C_{k,0}$ à $X_k$ en divisant par l'idéal $\idealx$. Idem avec
les prime. Et que l'on dispose d'identifications canoniques
$$
Y'_k = \rmK^{n-k}(\ux;\bbA/\idealap),\quad Y_k = \rmK^{n-k}(\ux;\bbA/\ideala),
\qquad\qquad
X'_k = \rmK^{n-k}(\ua';\bbA/\idealx),\quad X_k = \rmK^{n-k}(\ua;\bbA/\idealx)
$$
Un schéma en guise de résumé:
$$
\xymatrix @M=0.5pc {
Y'_\sbullet = \rmK^{n-\sbullet}\big(\ux \,; \bbA/\idealap \big) & & 
T'_\sbullet \ar@{->>}[ll] \ar@{->>}[dl] \\
& X'_\sbullet = \rmK^{n-\sbullet}\big(\ua' \,; \bbA/\idealx \big) & \\
&& \\ 
Y_\sbullet = \rmK^{n-\sbullet}\big(\ux\,;\bbA/\ideala\big) \ar@{<-}@[blue][uuu]_{\times\det(V)} 
& & 
T_\sbullet \ar@{->>}[ll] \ar@{->>}[dl]
          \ar@{<-}@[blue][uuu]_{T_\sbullet(v) = \bigwedge (v^{2n-\sbullet}\oplus\Id)} \\
& X_\sbullet = \rmK^{n-\sbullet}\big(\ua \,; \bbA/\idealx \big)
          \ar@{<-}@[blue][uuu]_{\bigwedge^{n-\sbullet}(v)\,\otimes\,\Id_{\bbA/\idealx}} & \\
}
$$
On ne retient que le degré homologique $n$ (c'est-à-dire le degré cohomologique Koszul $0$). 
Voici ce qui se passe au niveau des axes des abscisses à gauche, et au niveau des ordonnées
à droite:
$$
\xymatrix@M=0.5pc @C=4pc{ 
X'_n = \rmK^0(\ua'\,;\bbA/\idealx)  \ar@[blue][d]_-{\times 1}  
& Y'_n = \rmK^0(\ux\,;\bbA/\idealap)  \ar@[blue][d]^-{\times\det(V)}  \\
X_n = \rmK^0(\ua\,;\bbA/\idealx)  & Y_n = \rmK^0(\ux\,;\bbA/\ideala) 
}
$$
Lorsque la suite $\ux$ est complètement sécante, le théorème du
bicomplexe fournit un diagramme commutatif au \emph{niveau homologique}:
$$
\xymatrix@M=0.5pc @C=4pc{  
\rmH_n(X'_\sbullet)=\Ann\big(\ua'\,;\bbA/\idealx\big)  \ar@[blue][d]_-{\times 1}\ar@{-->}[r] 
 &\rmH_n(Y'_\sbullet) = \Ann\big(\ux\,;\bbA/\idealap\big)  \ar@[blue][d]^-{\times\det(V)}
\\
\rmH_n(X_\sbullet) =
\Ann\big(\ua \, ; \bbA/\idealx \big)   \ar@{-->}[r] & \rmH_n(Y_\sbullet) =\Ann\big(\ux\,;\bbA/\ideala\big) 
}
$$
Si de plus $\ua$ et $\ua'$ sont complètement sécantes, les horizontales pointillées sont 
des isomorphismes. C'est l'énoncé du théorème \ref{Vfonctorialite}.

\subsection {Algèbres graduées alternées et produit tensoriel gauche}
\label{AlgGradAlt}

Nous avons mentionné en page~\pageref{EncadrePrecision} le complexe de Koszul $\rmK^\sbullet(\ub,\uc)$
comme autre modèle du produit tensoriel $\rmK^\sbullet(\ub)\otimes\rmK^\sbullet(\uc)$, les
deux étant reliés par un isomorphisme canonique noté~$(\heartsuit)$.

\smallskip

Si on veut faire savant, cet autre modèle est lié, pour deux
$\bbA$-modules quelconques $M,M'$, à l'isomorphisme canonique dans la
catégorie des algèbres graduées alternées, isomorphisme que nous
notons de la même manière:
$$
\BW^\sbullet(M\oplus M')  \overset{\rm can.}{\simeq}
\BW^\sbullet(M) \ \gotimes\ \BW^\sbullet(M')
\leqno(\heartsuit)
$$
Le produit tensoriel $\gotimes$ qui intervient est nommé ``produit
tensoriel gauche''\footnote{La traduction anglaise de Bourbaki utilise
``skew tensor product''. On peut se demander si gauche à la
signification ``contraire de droite'' ou bien la signification
``tordu''? Dit autrement, est-ce qu'il existe un produit tensoriel
droit dans une catégorie adéquate d'algèbres? } par Bourbaki
(Alg. III, \S4, n$^\circ$7, p. A~III.49).  La
définition d'algèbre graduée alternée (\S4, n$^\circ$9, définition 7,
p. A~III.53) est moulée sur celle de l'algèbre extérieure (toute
algèbre extérieure en est donc une!). Nous la donnons ci-après.

\smallskip

Comme l'algèbre extérieure est l'objet de base sur lequel est monté le
complexe de Koszul, il est bon de rappeler la propriété universelle de
$\BW^\sbullet(M)$ où $M$ est un $\bbA$-module quelconque. Pour
n'importe quelle $\bbA$-algèbre $E$ et application $\bbA$-linéaire
$\varphi : M \to E$ vérifiant $\varphi(x)^2 = 0$ pour tout $x\in M$,
il existe un unique morphisme de $\bbA$-algèbres $\widetilde\varphi :
\BW^\sbullet(M) \to E$ qui prolonge $\varphi$. C'est essentiellement
sur cette propriété que repose, pour une application linéaire $v :
M_1\to M_2$, le fondement de $\BW^\sbullet(v) : \BW^\sbullet(M_1) \to
\BW^\sbullet(M_2)$.

\begin {defn} [Algèbre graduée alternée] \leavevmode

Une $\bfA$-algèbre graduée alternée $E$ est une $\bfA$-algèbre $\bbN$-graduée
vérifiant les deux propriétés
$$
\left\{
\begin {array}{ll}
x.y = (-1)^{\deg x\deg y}\, y.x &\text{pour tous éléments homogènes $x,y$}
\\
x^2 = 0  &\text{pour tout élément homogène $x$ de degré \emph{impair}}
\\
\end {array}
\right.
$$
\end {defn}

Le produit tensoriel gauche~$E \gotimes E'$ de deux telles algèbres
est défini de la manière suivante pour des éléments homogènes $x,y \in
E$, $x',y' \in E'$
$$
(x\otimes x'). (y \otimes y') = (-1)^{\deg x'\deg y}\, xy\otimes x'y'
$$
Pour disposer de cette définition écrite noir sur blanc, il est préférable
d'aller voir à l'intérieur de la preuve de la proposition 14: cela
dispense de lire la totale sur les facteurs de commutation. Cette
proposition~14 est celle qui prouve que le produit tensoriel gauche
de deux algèbres graduées alternées est une algèbre graduée alternée:
normal, tout a été élaboré pour qu'il en soit ainsi.

\index{produit tensoriel!gauche (algèbres graduées alternées)}%

\smallskip

Notons que dans la définition, la clause ``degré \emph{impair}'' est
capitale comme on le voit sur l'exemple suivant: $(e_1 \wedge e_2 + e_3\wedge
e_4)^2 = 2 e_1\wedge e_2\wedge e_3\wedge e_4$ où les $(e_i)$ forment
la base canonique de $\bbA^4$.

\medskip

Si on veut y comprendre quelque chose, il faut penser à l'algèbre
extérieure $\BW^\sbullet(M\oplus M')$ et au fait que $M,M'$ se
plongent dans $M\oplus M'$.  Supposons un instant, histoire de montrer
que la structure $\BW^\sbullet(M\oplus M')$ est \og pilotée\fg{} par
celles de $\BW^\sbullet(M)$ et $\BW^\sbullet(M')$, disposer d'éléments
homogènes $\bfu, \bfv$ dans $\BW^\sbullet(M)$ et $\bfu', \bfv'$ dans
$\BW^\sbullet(M')$ et qu'on veuille ramener le produit $(\bfu \wedge
\bfu') \wedge (\bfv \wedge \bfv')$ à un produit de la forme
$\text{truc} \wedge \text{truc'}$ avec $\text{truc} \in
\BW^\sbullet(M)$ et $\text{truc'} \in \BW^\sbullet(M')$.  Il suffit
pour cela d'utiliser la règle de commutation:
$$
(\bfu \wedge \bfu') \wedge (\bfv \wedge \bfv') =
(-1)^{\deg\bfu'\deg\bfv}
(\bfu \wedge \bfv) \wedge (\bfu' \wedge \bfv') 
$$
C'est sans doute ce petit calcul extérieur qui est à l'origine du produit
tensoriel gauche de deux algèbres graduées alternées.

\subsubsection*{Propriété universelle du produit tensoriel gauche}

C'est le troisième point de la proposition 10 de Bourbaki
(Alg. III, \S4, n$^\circ$7, p. A~III.47).  Soit $F$ une $\bbA$-algèbre
et deux morphismes d'algèbres $f : E \to F$, $f' : E' \to F'$ vérifiant:
$$
f(x)f'(x') = (-1)^{\deg x \deg x'} f'(x')f(x)
\qquad
\text{pour $x\in E$, $x'\in E'$ homogènes}
$$
Alors, en définissant $\iota : E \to E\gotimes E'$ par $x \mapsto x\otimes 1$ et
$\iota' : E' \to E\gotimes E'$ par $x' \mapsto 1\otimes x'$, il existe un et
un seul morphisme $E\gotimes E' \to F$ qui fait commuter le diagramme suivant:
$$
\xymatrix {
E\ar[dr]_\iota\ar[drrr]^f \\
   &E\gotimes E' \ar@{-->}[rr] && F \\  
E'\ar[ur]^{\iota'}\ar[urrr]_{f'} \\
}
$$

\subsubsection*{Quelques considérations informelles}

Nous allons d'abord examiner de manière informelle
la composante homogène de degré~$r$ en changeant provisoirement le
nom du module $M'$ en $N$:
$$
\BW^r(M \oplus N) \overset{\rm can.}{\simeq}
\bigoplus_{i+j=r} \BW^i(M) \otimes \BW^j(N)
\leqno(\heartsuit_r)
$$
Il est indispensable de rendre opérationnel ce $\overset{\rm
  can.}{\simeq}$.  
Histoire de se
rassurer, lorsque $M,N$ sont libres de dimensions respectives $m,n$,
en égalant les dimensions des modules libres de part et d'autre de
l'isomorphisme, on en tire:
$$
\binom{m+n}{r} = \sum_{i+j=r} \binom{m}{i}\binom{n}{j}
\qquad\text{que l'on peut obtenir via}\qquad
(X+Y)^{m+n} = (X+Y)^m (X+Y)^n
$$
D'ailleurs, l'isomorphisme canonique $(\heartsuit_r)$
peut être rapproché de la structure de l'ensemble des parties de cardinal donné
d'un ensemble. Ci-dessous, $\bfm, \bfn$ sont deux ensembles disjoints,
$\bfm \sqcup \bfn$ désigne leur réunion (disjointe) et 
$\calP_i(\bfm) \veebar \calP_j(\bfn)$ l'ensemble des
$X \sqcup Y$ avec $X \in \calP_i(\bfm)$ et $Y\in \calP_j(\bfn)$:
$$
\calP_r(\bfm \sqcup \bfn) = \bigsqcup_{i+j=r} \calP_i(\bfm)\veebar\calP_j(\bfn)
$$
Revenons aux modules en supposant $M$ libre de dimension $m$ avec une base $(e_1, \dots, e_m)$
et $N$ libre de dimension $n$ avec une base $(f_1, \dots, f_n)$.  Pour
$I \subset \{1..m\}$ de cardinal $i$ et $J \subset \{1..n\}$ de cardinal
$j$ avec $i+j = r$, l'isomorphisme canonique $(\heartsuit_r)$ fait correspondre
$e_I\wedge f_J$ et $e_I\otimes f_J$.  On obtient ainsi une base
$(e_I \wedge f_J)_{\#I+\#J=r}$ de $\BW^r(M \oplus N)$ adaptée à la
somme directe intervenant dans cet isomorphisme canonique.
D'ailleurs, en posant $\bfm = \{1..m\}$, en indexant la base de $N$
par $\bfn := \{m+1..m+n\}$ au lieu de $\{1..n\}$ et en la notant
$(e_{m+1}, \dots, e_{m+n})$, on obtient $e_I \wedge e_J = e_{I \sqcup J}$.
D'où le rapprochement avec l'ensemble des parties.

\subsubsection*{Une manière formelle de définir $(\heartsuit)$ et son inverse.}

$\blacktriangleright$
Dans le sens $\BW^\sbullet(M\oplus M') \leftarrow \BW^\sbullet(M)\gotimes\BW^\sbullet(M')$.
Notons $\iota : M \to M\oplus M'$ et $\iota' : M \to M\oplus M'$ les injections
canoniques. Elles donnent naissance aux deux morphismes:
$$
\BW^\sbullet(\iota) : \BW^\sbullet(M) \to \BW^\sbullet(M\oplus M')
\qquad\qquad
\BW^\sbullet(\iota') : \BW^\sbullet(M') \to \BW^\sbullet(M\oplus M')
$$
En appliquant la propriété universelle du produit tensoriel gauche à ces deux morphismes,
on obtient un morphisme dans le sens désiré. Il réalise:
$$
\BW^\sbullet(\iota)(\bfu)\ \wedge\ \BW^\sbullet(\iota')(\bfv)\
\longleftarrow
\bfu \otimes \bfv 
$$
De manière un peu plus légère, dans l'isomorphisme~$(\heartsuit_r)$,
pour $i+j=r$, la flèche $\BW^r(M\oplus M') \leftarrow \BW^i(M)\otimes\BW^j(M')$
est fournie par:
$$
u_1 \wedge\cdots\wedge u_i\ \wedge\ v_1 \wedge\cdots\wedge v_j
\longleftarrow
(u_1 \wedge\cdots\wedge u_i) \otimes (v_1 \wedge\cdots\wedge v_j) 
$$
I.e. pour $\bfu\in\BW^i(M)$ et $\bfv \in \BW^j(M')$, l'isomorphisme
$(\heartsuit_r)$ réalise $\bfu\wedge\bfv \longleftrightarrow \bfu\otimes\bfv$.

\medskip

$\blacktriangleright$
Dans le sens $\BW^\sbullet(M\oplus M') \to
\BW^\sbullet(M)\gotimes\BW^\sbullet(M')$.  La composante homogène de
degré 1 de l'algèbre graduée alternée à l'arrivée est $(M\otimes \bbA)
\oplus (\bbA\otimes M')$, \emph {canoniquement} isomorphe à $M\oplus
M'$. Notons $\varphi : M \oplus M' \to (M\otimes \bbA) \oplus
(\bbA\otimes M')$ cet isomorphisme canonique. Pour $x+ x' \in M\oplus
M'$, comme $\varphi(x + x')$ est homogène de degré 1, on a $\varphi(x
+ x')^2 = 0$.  D'après la propriété universelle de l'algèbre
extérieure, $\varphi$ se prolonge de manière unique en un morphisme
d'algèbres $\widetilde\varphi : \BW^\sbullet(M\oplus M') \to
\BW^\sbullet(M)\gotimes\BW^\sbullet(M')$.

\smallskip

Il resterait bien sûr à vérifier que ces deux morphismes d'algèbres sont réciproques
l'un de l'autre.

\medskip

Il ne nous semble pas inutile d'apporter quelques détails sur l'injection linéaire $\varphi$
et son prolongement~$\widetilde\varphi$ à $\BW^\sbullet(M\oplus M')$.
$$
\varphi :
\begin {array}[t]{rcl}
M\oplus M' &\to & (M\otimes\bbA) \oplus (\bbA\otimes M') \ \subset\ \BW^\sbullet(M)\gotimes\BW^\sbullet(M')
\\ [2mm]  
\qquad
x + x' &\mapsto &\quad x\otimes 1 \ +\  1\otimes x'
\\
\end {array}
$$
Amusons-nous (sic) à (re)-vérifier $\varphi(x + x')^2 = 0$ pour tout $x+ x' \in M\oplus M'$.
$$
\begin {array}{ccl}
(x\otimes 1 + 1\otimes x')^2
&=& 
  (x\otimes 1).(x\otimes 1) + (x\otimes 1).(1\otimes x') +
  (1\otimes x').(x\otimes 1) + (1\otimes x').(1\otimes x')
 \\
&=& (x\wedge x)\otimes 1 + (x\otimes x' - x\otimes x') + 1\otimes(x'\wedge x')
\\
\end {array}
$$
Cela fait bien 0! Voici un exemple de ce que réalise le prolongement $\widetilde\varphi$ de $\varphi$:
$$
(x_1 + x'_1)\wedge (x_2+x'_2) \mapsto
(x_1\wedge x_2)\otimes 1 \ +\  (x_1\otimes x'_2 - x_2\otimes x'_1) \ +\  1\otimes (x'_1 \wedge x'_2)
$$
L'image par $\widetilde \varphi$ est une somme $c_{2,0} + c_{1,1} + c_{0,2}$ où chaque $c_{p,q}$ est bihomogène
de bidegré $(p,q)$. Plus généralement:
$$
\widetilde\varphi : \bigwedge_{i=1}^r (x_i + x'_i) \mapsto \sum_{p+q=r} c_{p,q}
$$
où chaque $c_{p,q}$ est bihomogène de bidegré $(p,q)$. En notant
$\fS'_p\subset \fS_r$ le sous-ensemble de cardinal $\binom{r}{p}$
constitué des permutations croissantes sur $\{1..p\}$ et sur $\{p+1..r\}$
(permutation $(p,q)$-shuffle):
$$
c_{p,q} = \sum_{\sigma\in\fS'_p} \varepsilon(\sigma)\,
(x_{\sigma(1)}\wedge\cdots\wedge x_{\sigma(p)}) \otimes (x'_{\sigma(p+1)}\wedge\cdots\wedge x'_{\sigma(r)}) 
$$
  
\subsubsection*{Back to Koszul}

Ce sont toutes ces considérations qui conduisent, en ce qui concerne
le complexe de Koszul montant $\rmK^\sbullet(a_1,\dots,a_n)$, à
prendre, en notant $a = \sum_i a_i e_i \in \bbA^n$, comme
différentielle $a\wedge\sbullet$ et non $\sbullet\wedge a$. Disons
plutôt que c'est ce côté \emph {gauche} que nous avons retenu et qui
nous semble présenter une certaine compatiblité avec la définition du
produit tensoriel de deux complexes (cf ci-après l'adéquation entre
différentielles).  En remarquant que cette différentielle est définie
à partir de la structure d'algèbre de $\BW^\sbullet(\bbA^n)$ (le
produit extérieur), lorsque l'on écrit par exemple que le complexe de
Koszul montant est canoniquement isomorphe au produit extérieur de
complexes de Koszul montants élémentaires:
$$
\rmK^\sbullet (a_1,\dots,a_n) \simeq \rmK^\sbullet(a_1)\otimes\cdots\otimes\rmK^\sbullet (a_n) 
$$
c'est sous-entendu que l'isomorphisme est un isomorphisme dans la catégorie des algèbres graduées
alternées et qu'à droite $\otimes$ désigne $\gotimes$. Et de ce fait, la \og structure différentielle montante\fg{}
suivra. I.e. l'isomorphisme canonique est de surcroît un isomorphisme de complexes (montants).

\smallskip

Une fois que l'on a opté pour $a\wedge\sbullet$ (produit
extérieur à gauche par $a$) comme différentielle sur le montant, on
est quasiment obligé (pour des raisons de dualité) de prendre
$\sbullet\intd a$ (produit intérieur droit par $a$) comme
différentielle sur le descendant.

\subsubsection*{Vérification de l'adéquation des différentielles énoncée
en page~\pageref{AdequationDifferentielles}}

\index{produit tensoriel!de complexes}%

Il faut vérifier l'adéquation entre la différentielle du produit tensoriel,
définie sur $\rmK^i(\ub) \otimes \rmK^j(\uc)$ par:
$$
\delta^\ub \otimes \Id + (-1)^i \Id\otimes \delta^\uc
$$
et la différentielle $\delta^{\ub,\uc}$ de $\rmK^{\ub,\uc}$ donnée par $(b+c)\wedge\sbullet$.

Soient $\bfu \in \BW^i \rmK^1(\ub)$ et $\bfv \in \BW^j\rmK^1(\uc)$,
en rappelant que $\bfu \otimes \bfv$ et $\bfu \wedge \bfv$ se
correspondent par l'isomorphisme $(\heartsuit)$:
$$
\begin {array} {rcl}
(b+c) \wedge (\bfu \wedge \bfv)
&=& b\wedge\bfu \wedge \bfv + c\wedge \bfu \wedge \bfv
\\
&=&
(b\wedge\bfu)\wedge\bfv + (-1)^i \bfu \wedge (c \wedge \bfv)
\\
&=&
\delta^\ub(\bfu) \wedge\bfv + (-1)^i \bfu \wedge \delta^\uc(\bfv)
\\
&=&
(\delta^\ub\otimes\Id)(\bfu \otimes \bfv) + (-1)^i (\Id\otimes\delta^\uc)(\bfu\otimes\bfv)
\\
\end {array}
$$
Bilan: dans~$(\heartsuit)$, la différentielle $\delta^{\ub,\uc}$ correspond au produit tensoriel
des différentielles $\delta^\ub$ et $\delta^\uc$.


\section{Annexe: le sous-module de Fitting $\FittVect_1(M) \subset M^\star$}

En~\ref{1FittingVectoriel}, nous avons énoncé le résultat suivant que
nous nous proposons de prouver. On rappelle que pour une application
linéaire $u : E \to F$ entre deux modules libres, on a défini (cf la
section~\ref {SectionDetCoRang1}), en posant $r = \dim F -1$, le
sous-module $\DVect_r(u) \subset F^\star$ comme l'image de
$\Im\BW^r(u)$ par un isomorphisme déterminantal
$\BW^r(F)\overset{\simeq}{\longrightarrow}F^\star$.  Ici, en notant
$\bff$ une orientation de~$F$, nous aurons un petit faible pour
l'isomorphisme déterminantal $\bfw \mapsto [\bfw \wedge \sbullet]_\bff$.

\index{isomorphisme déterminantal}%
\begin{theo}[Définition du sous-module de Fitting en \og co-rang\fg{} 1] 
\leavevmode

Soit $M$ un module de présentation finie vérifiant $\calF_0(M) = 0$. 
On considère $u : E \to F$ et $u' : E' \to F'$ deux présentations de $M$.
On pose $r = \dim F - 1$ et $r'= \dim F' - 1$.
Alors $\DVect_r(u)$ et $\DVect_{r'}(u')$ vus dans~$M^\star$ sont égaux. 

De manière plus formelle, 
on a l'égalité des images réciproques :
$(\transpose {\pi})^{-1} \big(\DVect_r(u)\big) 
\ = \ 
(\transpose {\pi'})^{-1}\big(\DVect_{r'}(u')\big)$.
$$
\vcenter 
{\xymatrix @R=2pt{
E\ar[r]^u    & F\ar@{->>}[dr]^{\pi}  \\
             &            &M \\
E'\ar[r]_{u'} & F'\ar@{->>}[ur]_{\pi'}  \\
}}
\hspace{3cm}
\vcenter 
{\xymatrix @R=2pt{
\llap{$\DVect_r(u) \ \subset \ $} F^\star \ar@{<-}[dr]^-{\transpose \pi} & \\
  & M^\star \\
\llap{$\DVect_{r'}(u') \ \subset \ $} F'^\star  \ar@{<-}[ur]_-{\transpose \pi'} & \\
}}
$$
Ce sous-module de $M^\star$, indépendant de la présentation, est noté $\FittVect_1(M)$.
\end{theo}


\noindent
\subsubsection*{Deux faits utilisés dans la preuve du théorème}

\noindent
$\bullet$
Fait 1 (lemme ensembliste dont la preuve ne pose pas de problème).

\begin{quote}
\it 
Supposons disposer du diagramme commutatif  \qquad 
$
\xymatrix @R=2pt @M=0.5pc{
\llap{$\mathscr A \ \subset \ $} A \ar@{^{(}->}[dd]_{\Phi} \ar@{<-}[dr]^-{\alpha} & \\
  & X \\
\llap{$\mathscr B \ \subset \ $} B \ar@{<-}[ur]_-{\beta} & \\
}
$

Si $\Phi$ est injective et $\mathscr B = \Phi(\mathscr A)$, 
alors ${\alpha}^{-1}(\mathscr A) \ = \ {\beta}^{-1}(\mathscr B)$.
\end{quote}

\bigskip

\noindent
$\bullet$ Fait 2

\begin{quote}
\it 
Soient $u : E \to F$ et $u' : E' \to F$ deux applications linéaires
vérifiant $\Im u' \subset \Im u$. 
Alors pour tout~$r$, on a $\DVect_r(u') \subset \DVect_r(u)$.
En particulier, si $u,u'$ ont même image, alors $\DVect_r(u') = \DVect_r(u)$.
\end{quote}

Justification. Soit $\bfw^\sharp$ une forme linéaire  de $\DVect_r(u')$. 
Le vecteur $\bfw$ est une combinaison linéaire de $y'_1 \wedge \cdots \wedge y'_r$ 
avec $y'_i \in \Im u'$.
Par hypothèse $\Im u' \subset \Im u$, donc $\bfw$ s'écrit comme combinaison linéaire 
de $y_1 \wedge \cdots y_r$ avec $y_i \in \Im u$, donc appartient à $\Im \bigwedge^r(u)$.
Ainsi $\bfw^\sharp$ est dans $\DVect_r(u)$.

\subsubsection*{La preuve du théorème}
\label{Preuve1FittingVectorielInvariance}

\noindent
L'hypothèse $\calF_0(M)$ se traduit par $\BW^{r+1}(u) = 0$ (idem avec prime).
On commence par prouver le théorème dans deux cas particuliers. 

\medskip

\textbf{\'Etape 1.} Modification élémentaire de la présentation par un module $L$ 
libre de dimension $\ell$.
$$
\vcenter 
{\xymatrix @C=1.2cm @R=2pt @M=0.5pc{
E\ar[r]^u    & F\ar@{->>}[dr]^{\pi}  \\
             &            & M \\
E \oplus L \ar[r]_{\widetilde u = u \oplus \id_L} & 
F \oplus L \ar[uu]_{\pi_F} \ar@{->>}[ur]_{0 \oplus \pi}  \\
}}
\hspace{3cm}
\vcenter 
{\xymatrix @C=1.2cm @R=2pt @M=0.5pc{
\llap{$\DVect_r(u) \ \subset \ $} F^\star \ar@{<-}[dr]^-{\transpose \pi} 
\ar@{^(->}[dd]_{\transpose \pi_F} & \\
  & M^\star \\
\llap{$\DVect_{r+\ell}(\widetilde u) \ \subset\ $}(F\oplus L)^\star
\ar@{<-}[ur]_-{\transpose (0 \oplus \pi)} & \\
}}
$$
La flèche descendante dans le diagramme de droite $\transpose \pi_F :
F^\star \to (F \oplus L)^\star$ est l'application qui consiste à
prolonger une forme linéaire sur $F$ par $0$ sur $L$ et est donc
injective.

\smallskip

Dans ces conditions, pour montrer que les modules $\DVect_r(u)$ et
$\DVect_{r+\ell}(\widetilde u)$ vus dans $M^\star$ sont égaux, il
suffit, d'après le fait~1, de montrer que $\transpose
\pi_F\big(\DVect_r(u)\big) = \DVect_{r+\ell}(\widetilde u)$.

Soit $[\bfw \wedge \sbullet]_\bff \in \DVect_r(u)$.
Prenons son image par $\transpose \pi_F$, ce qui revient à prolonger par $0$ sur $L$ cette forme linéaire.
Alors cette forme linéaire prolongée s'écrit $[\bfw \wedge \sbullet\wedge \bfl]_{\bff \wedge \bfl}$
où $\bfl$ est une orientation de $L$.
Cette dernière forme linéaire est donc dans $\DVect_{r+\ell}(\widetilde u)$.

Réciproquement, prenons une forme linéaire du type $\widetilde \bfw^\sharp$ 
avec $\widetilde \bfw = \bigwedge^{r+\ell}(\widetilde u)(x \oplus \bfl)$.
Ici on utilise l'hypothèse $\calF_0(M) = 0$ qui se traduit par
$\bigwedge^{r+1}(u) = 0$.
Comme de plus $\bigwedge^{\ell+1}(\id_L) = 0$, 
cette forme linéaire vaut $[\bfw \wedge \bfl \wedge \sbullet]_{\bff \wedge \bfl}$
où $\bfw = \bigwedge^{r}(u)(x)$.
Elle est donc égale au prolongement de $\pm [\bfw \wedge \sbullet]_\bff$.

\medskip

\textbf{\'Etape 2.}
On suppose disposer de deux présentations \og concordantes \fg{} dans le sens 
où il existe deux isomorphismes $\Phi$ et $\Phi'$ réciproques l'un de l'autre 
rendant commutatif le diagramme de gauche :
$$
\vcenter 
{\xymatrix @R=2pt @M=0.5pc{
G \ar[r]^v & H \ar@<-1ex>[dd]_{\Phi} \ar@{->>}[dr]^{\pi}  \\
             &            & M \\
G' \ar[r]_{v'} & H' \ar[uu]_{\Phi'} \ar@{->>}[ur]_{\pi'}  \\
}}
\hspace{3cm}
\vcenter 
{\xymatrix @R=2pt @M=0.5pc{
\llap{$\DVect_s(v) \ \subset \ $} H^\star \ar@{<-}[dr]^-{\transpose \pi} 
\ar[dd]_{\transpose \Phi'} & \\
  & M^\star \\
\llap{$\DVect_{s'}(v') \ \subset\ $} H'^\star
\ar@{<-}[ur]_-{\transpose \pi'} & \\
}}
\qquad
\begin {array}{c}
s = \dim H - 1 \\
\\
\\
s' = \dim H' - 1 \\
\end {array}
$$
Bien entendu, il suffit que $\Phi$ soit un isomorphisme et que
$\pi=\pi'\circ\Phi$. Nous allons montrer
$$
\DVect_s(v) = \DVect_{s'}(v') \qquad \text{vus dans $M^\star$}
$$
Il suffit, d'après le fait~1, de montrer que 
$\transpose \Phi'\big(\DVect_{s}(v)\big) = \DVect_{s'}(v')$
et pour cela, de prouver:
$$
\DVect_{s'}(\Phi \circ v) \ = \  \DVect_{s'}(v'),
\qquad\qquad
\transpose \Phi'\big(\DVect_s(v)\big)  \ = \  \DVect_{s'}(\Phi \circ v)
$$
L'égalité de gauche s'obtient en utilisant le fait~2 et $\Im(\Phi
\circ v) = \Im(v')$, égalité que nous justifions de la manière
suivante. Par commutation, on a $\Phi(\Ker \pi) \subset \Ker \pi'$ et
$\Phi'(\Ker \pi') \subset \Ker \pi$; puisque $\Phi$, $\Phi'$ sont
réciproques, il vient $\Phi(\Ker \pi) = \Ker \pi'$ c'est-à-dire
$\Phi(\Im v) = \Im v'$.

\smallskip
Pour vérifier l'égalité de droite, prenons une forme linéaire dans
l'ensemble de gauche de la forme $[\bfw \wedge
  \Phi'(\sbullet)]_{\bfh}$ où $\bfw \in \Im \bigwedge^s(v)$ et $\bfh$
une orientation de $H$ et prouvons qu'elle est dans l'ensemble de
droite (l'inclusion opposée s'en déduit en remplaçant $v$ par $\Phi'
\circ v$ et en appliquant $\transpose \Phi$ qui est bijectif).  Soit
$\bfh'$ une orientation de $H'$.  En utilisant $\bfw =
\big(\bigwedge^s \Phi' \circ \bigwedge^s \Phi \big)(\bfw)$, on obtient:
$$
\big[\bfw \wedge \Phi'(\sbullet)\big]_{\bfh}  \ = \
\det\nolimits_{\bfh',\bfh}(\Phi') \,\cdot\,
\big[\textstyle \bigwedge^s(\Phi)(\bfw) \wedge \sbullet\big]_{\bfh'}
$$
Comme $\Phi'$ est un isomorphisme, le déterminant ci-dessus est inversible, 
et donc la forme linéaire dont on est parti 
s'écrit bien $[\bfw' \wedge \sbullet]_{\bfh'}$ où $\bfw' \in \Im \bigwedge^s (\Phi \circ v)$.

\medskip

\textbf{Conclusion.}
Une fois ces deux cas particuliers assurés,
le lemme de Schanuel dans sa version améliorée (cf l'énoncé
au début du chapitre~\ref{ChapFittingVectoriel}) va nous
permettre de conclure en fournissant deux présentations
\og concordantes\fg{} au sens de l'étape~2.

Reprenons le contexte général du théorème à prouver. En posant
$\widetilde u = u \oplus \id_{F'}$ et $\widetilde u' = \id_F \oplus
u'$, le lemme de Schanuel fournit deux isomorphismes $\Phi,
\Phi'$ réciproques l'un de l'autre rendant commutatif:
$$
\xymatrix @R=2pt @M=0.5pc{
E \oplus F' \ar[r]^{\widetilde u} & F \oplus F' \ar@<-1ex>[dd]_{\Phi} \ar@{->>}[dr]^{\pi \oplus 0}  \\
             &            & M \\
F \oplus E' \ar[r]_{\widetilde u'} & F \oplus F' \ar[uu]_{\Phi'} \ar@{->>}[ur]_{0\oplus\pi'}  \\
}
$$
Avec $H = H' = F \oplus F'$, on peut appliquer l'étape 2 qui affirme
(en notant $f = \dim F$ et $f' = \dim F'$) :
$$
\DVect_{r+f'}(\widetilde u) \ = \ \DVect_{r' + f}(\widetilde u') 
\qquad \text{vus dans $M^\star$}
$$ 
Et maintenant, on utilise l'étape 1 pour les deux membres, d'où 
$$
\DVect_r(u) \ = \ \DVect_{r'}(u') 
\qquad \text{vus dans $M^\star$}
$$ 
ce qu'il fallait démontrer.


\section{Algèbre linéaire déterminantale}

\noindent
\fbox{\parbox{0.97\linewidth}{%
Dans cette annexe, on considère une application linéaire $u : E \to F$
où $E$, $F$ désignent des modules libres de rang fini.  On notera
$(e_j)$ une base ordonnée de $E$ et $(f_i)$ une base ordonnée de $F$.

\smallskip

Interviendra également un entier $r$.  On propagera les bases ci-dessus
aux puissances extérieures $\BW^r(E)$ et $\BW^r(F)$ en
$(e_J)_{\# J = r}$ et $(f_I)_{\# I = r}$
$$
e_J = \bigwedge_{j \in J} e_j,
\qquad\qquad
f_I = \bigwedge_{i \in I} f_i
$$
On réservera les lettres $J$ et $I$: $J$ désignera une partie de
cardinal $r$ de l'ensemble des indices de la base de $E$.  Idem pour
$I$ avec la base de $F$.

\medskip
Pour une telle partie $I$, on note:
$$
[\, \dots \ ]\strut_I : \overbrace{F \times \cdots \times F}^{r} \to \bfA
$$ 
la forme $r$-linéaire alternée qui à $(y_1, \dots, y_r)$ associe le
déterminant $[y_1, \dots, y_r ]\strut_I$ sur les lignes d'indices~$I$
des vecteurs $y_1, \dots, y_r$ de~$F$ exprimés dans la base $(f_i)$.
Autrement dit, $[y_1, \dots, y_r ]\strut_I$ est la composante de $y_1
\wedge \cdots \wedge y_r$ sur $f_I$.
}}

\bigskip

Dans cette annexe, nous fournissons, via l'algèbre extérieure, quelques
compléments d'algèbre linéaire qui mettent en jeu une application
linéaire $u : E \to F$ entre deux modules libres et font intervenir
ses mineurs et la factorisation de sa puissance extérieure ad-hoc.
Nous fournirons une généralisation de la formule de Cramer permettant
de retrouver certains ingrédients de~\cite[II, section
  5]{LombardiQuitte}, en particulier la fameuse \og navette
matricielle\fg{} $ABA = A$.  Nous compléterons également les résultats
de structure locale (au sens profondeur $\ge 2$) des modules de MacRae
de rang $c$ étudiés au chapitre~\ref{ChapFittingVectoriel}.  Nous
verrons comment l'angle de la factorisation, objet de
la proposition~\ref{FactorisationEtPgcdFort}, permet un traitement
efficace.

Enfin, comme corollaires de cette étude, nous fournirons les résultats
nécessaires au traitement des idéaux de factorisation d'un complexe
exact libre, celui concernant le ``next unexpected result'' de
Northcott (cf. le th.~\ref{RacineIdealFactorisation} de l'avant-dernière
section du chapitre~\ref{ChapStructureMultiplicative}).

\bigskip

Nous commençons par l'énoncé suivant ne faisant intervenir
ni mineurs ni factorisation.

\begin{prop}[Caractérisation des applications linéaires d'image facteur direct]\label{ImFacteurDirect}
On dispose de l'équivalence
\begin {enumerate}[\rm i)]
\item  
$\Im u$ est facteur direct dans $F$.

\item  
Il existe $v : F \to E$ tel que $u = u \circ v \circ u$. 
\end {enumerate}

\smallskip
Auquel cas, $\Ker u$ est facteur direct dans $E$.
\end {prop}

\begin {proof} \leavevmode

i) $\Rightarrow$ ii)
Le sous-module $\Im u$ étant facteur direct dans $F$, il existe un
projecteur $\pi$ de $F$ d'image~$\Im u$.  En conséquence, pour chaque
$f_i$, il existe $x_i \in E$ tel que $\pi(f_i) = u(x_i)$.  On définit
$v : F \to E$ par $v(f_i) = x_i$.  On a alors $u \circ v = \pi$. Donc
$u \circ v \circ u = \pi \circ u = u$.

ii) $\Rightarrow$ i)
L'endomorphisme $u \circ v : F \to F$ est un projecteur de $F$ de même
image que $u$.  Donc $\Im u$ est facteur direct dans $F$.

\smallskip
En ce qui concerne le noyau: de manière analogue à $\Im(u\circ v) = \Im(u)$,
on a $\Ker(u) = \Ker(v\circ u)$ et $v\circ u$
est un projecteur de $E$ de sorte que $\Ker(u)$ est facteur direct dans $E$.
\end {proof}

\subsubsection*{Structuration/organisation des mineurs: Cramer a encore frappé}

Nous avons commencé la première section du chapitre~\ref{ChapAC1} en
rappelant la formule dite de Cramer. Celle-ci fait intervenir une
matrice carrée en dimension $n$ et un vecteur additionnel de
$\bfA^n$. Si on y regarde de près, on voit que la matrice est de peu
d'importance, ce qui compte surtout ce sont ses $n$ colonnes. On peut
donc considérer qu'en donnée on dispose de $n+1$ vecteurs de $\bfA^n$
et que la formule de Cramer s'énonce ainsi, par exemple pour $n=3$
(en désignant par $\det$ le déterminant dans la base canonique):
$$
\det(y_1,y_2,y_3)\,y = \det(y,y_2,y_3)\,y_1 + \det(y_1,y,y_3)\,y_2 + \det(y_1,y_2,y)\,y_3
$$
Si on y regarde d'encore plus près, on constate que la raison profonde de
cette égalité réside dans la nullité du produit extérieur des $n+1$
vecteurs donnés de $\bfA^n$. Ce qui nous conduit à
l'énoncé\footnote{%
Selon le dictionnaire de l'Académie Française, ``scholie'', que l'on écrit aussi ``scolie'', désigne
en mathématiques ou logique, une remarque complémentaire
faite à propos d'un théorème, d'une proposition.
Pourquoi avoir utilisé ce mot dans un contexte qui n'est probablement pas approprié?
La raison en est simple: cela permet à l'auteur de le localiser parmi
les théorèmes, propositions, lemmes, faits: ici, c'est le
seul endroit où il est utilisé.
}
plus général suivant.

\begin {scholie} [Encore une formule de Cramer] \label {CramerScholie}

Etant donné un module libre $L$ et $r+1$ vecteurs $y, y_1, \ldots,
y_r$ de $L$, on dispose des équivalences:

\begin {enumerate}[\rm i)]
\item
$y \wedge y_1 \wedge \cdots \wedge y_r = 0$

\item
Pour toute forme $(r+1)$-linéaire alternée $g$ sur $L$, 
on a $g(y,y_1,\ldots,y_r) = 0$.

\item
Pour toute forme $r$-linéaire alternée $f$ sur $L$:
$$
f(y_1,\ldots,y_r)\,y = \sum_{i=1}^r\ f(y_1, \ldots, y_{i-1},\ y,\ y_{i+1}, \ldots, y_r)\, y_i
$$
\end {enumerate}
\end {scholie}

\begin {proof} \leavevmode

Via la correspondance biunivoque canonique entre 
formes linéaires sur~$\BW^m(L)$ et formes $m$-linéaires
alternées sur~$L$, on se permettra, dans la preuve de
la dernière implication, d'identifier une forme du
premier espace à une forme du second.

\medskip  
$\rm i) \Longleftrightarrow ii)$
Puisque $\BW^{r+1}(L)$ est un module libre, on a l'équivalence
$$
y \wedge y_1 \wedge \cdots \wedge y_r = 0  \iff
\big[\ \forall\nu \in \BW^{r+1}(L)^\star \quad \nu(y \wedge y_1 \wedge \cdots \wedge y_r) = 0\ \big]
$$
D'où le résultat en utilisant la correspondance biunivoque canonique évoquée ci-dessus.

\medskip  
$\rm ii) \implies iii)$
Soit $f$ une forme $r$-linéaire alternée sur $L$. Etant donnée une forme linéaire $\mu$ sur $L$,
on définit $\mu\wedge f : L^{r+1} \to \bfA$ via:
$$
\begin {array} {rcl}
(\mu\wedge f)(z_0, z_1,\ldots,z_r)
&=&
\sum_{i=0}^r (-1)^i \mu(z_i)\,f(z_1,\ldots,z_{i-1},z_{i+1},\ldots, z_r)
\\[0.3cm]
&=&
\mu(z_0)\,f(z_1,\ldots,z_r) - \sum_{i=1}^r \mu(z_i)\,f(z_1,\ldots,z_{i-1},\ z_0,\ z_{i+1},\ldots,z_r)
\\[0.3cm]
&=&
\mu(v) \quad\text{où}\quad
v = f(z_1,\ldots,z_r)\,z_0 - \sum_{i=1}^r f(z_1,\ldots,z_{i-1},\ z_0,\ z_{i+1},\ldots,z_r)\,z_i
\\
\end {array}
$$
On démontre classiquement, via la première expression, que $\mu\wedge f$ est une forme
$(r+1)$-linéaire alternée. Donc, par hypothèse, $(\mu\wedge f)(y, y_1,\ldots,y_r) =0$.
Ainsi, en notant
$$
w = f(y_1,\ldots,y_r)\,y - \sum_{i=1}^r\ f(y_1, \ldots, y_{i-1},\ y,\ y_{i+1}, \ldots, y_r)\,y_i
$$
on a $\mu(w) = 0$ pour tout $\mu \in L^\star$ d'où $w=0$, ce qui est le résultat à montrer.

\medskip  
$\rm iii) \implies ii)$
Par hypothèse, on a $(\mu\wedge f)(y,y_1,\ldots,y_r) = 0$ pour toute
forme $r$-linéaire alternée $f$ et toute forme linéaire $\mu$. Il
suffit donc de montrer que toute forme $(r+1)$-linéaire alternée est
une combinaison linéaire de formes $\mu\wedge f$. Modulo
l'identification mentionnée en début de preuve, il suffit de le voir
pour une base $(e^\star_I)_{\#I=r+1}$ de
$\BW^{r+1}(L^\star)\overset{\rm can.}{\simeq} \BW^{r+1}(L)^\star$ déduite
d'une base $(e_i)$ de $L$. Or, chaque~$e^\star_I$ est déjà de la forme
$\mu\wedge f$; par exemple pour $r=3$: $e^\star_{1234} = \mu\wedge f$
avec $\mu = e^\star_1$, $f=e^\star_{234}$.
\end {proof}

Venons-en à l'égalité $\calD_{r+1}(u) = 0$ qui évoque la nullité d'une
floppée de mineurs.  Il est parfois préférable de la structurer de la
manière suivante: on l'écrit de manière équivalente
$\BW^{r+1}(u) = 0$ et on la traduit en termes de relations
linéaires entre les colonnes d'une matrice $U$ de~$u$.  Par exemple
l'égalité $\BW^2(u) = 0$ est équivalente aux relations
$u_{i,j} U_k = u_{i,k} U_j$ pour toute paire $U_j$, $U_k$ de colonnes
et tout indice de ligne $i$.

Une conséquence de la formule générale de Cramer précédente est le corollaire
suivant qui explicite ces relations linéaires.

\begin{coro}[Une variante de Cramer]\label{ExtPowerNullity}
Supposons $\BW^{r+1}(u) = 0$.
Alors pour toute partie $I$ de cardinal $r$,
tous ${x_1, \dots, x_r \in E}$ et $x \in E$, on a :
$$
\big[u(x_1),\dots, u(x_r)\big]\strut_I \ 
u(x) 
\ =\ 
\sum_{\ell=1}^r \, 
\big[u(x_1), \dots, u(x_{\ell-1}), u(x), u(x_{\ell+1}), \dots, u(x_r)\big]\strut_I \ 
u(x_\ell)
$$
\end{coro}

\begin{prop}[Rang fort et image facteur direct]\label{RangFort}
On suppose $u$ de rang fort égal à~$r$ \idest{} $1 \in \calD_r(u)$ et
$\calD_{r+1}(u) = 0$.  Alors $\Im u$ est facteur direct dans $F$ (et
donc $\Ker u$ est facteur direct dans~$E$).
\end{prop}

\begin{proof}

On va construire $v : F \to E$ comme dans la proposition~\ref{ImFacteurDirect}
i.e. vérifiant $u \circ v \circ u = u$.

\smallskip
\noindent
$\rhd$
Ici, on utilise uniquement l'égalité $\BW^{r+1}(u) = 0$.
Soit $I$ de cardinal $r$ et $x_1, \ldots, x_r \in E$. On définit
$w : F \to E$ (sans mention de la dépendance en $I$ et $x_1, \ldots, x_r$) par 
$$
\forall\, y \in F, \qquad 
w(y) \ =\  
\sum_{\ell=1}^r\, 
\big[u(x_1), \dots, u(x_{\ell-1}),\ y,\ u(x_{\ell+1}), \dots, u(x_r)\big]\strut_I\  x_\ell
$$
Grâce au corollaire~\ref{ExtPowerNullity} on obtient une expression de $u \circ w \circ u$
$$
\forall\, x \in E, \qquad 
\begin {array}[t]{ccl}
(u \circ w \circ u)(x) &=&
   \displaystyle\sum_{\ell=1}^r\, 
   \big[u(x_1), \dots, u(x_{\ell-1}),\ u(x),\ u(x_{\ell+1}), \dots, u(x_r)\big]\strut_I\, u(x_\ell) 
\\ [0.4cm]
& = & 
\big[u(x_1), \dots, u(x_{\ell-1}),\ u(x_\ell),\ u(x_{\ell+1}), \dots, u(x_r)\big]\strut_I\, u(x)
\\
\end {array}
$$
On a donc $u \circ w \circ u =  [u(x_1),\dots, u(x_r)]\strut_I\, u$.

\smallskip
\noindent
Soit $J = \{j_1 < \ldots < j_r\}$ une partie de cardinal $r$.
Appliquons ce qui précède à $(x_1, \ldots, x_r) = (e_{j_1}, \ldots,
e_{j_r})$ et notons $v_{I,J} : F \to E$ l'application $w$. Alors
$$
u \circ v_{I,J} \circ u \ =\  \det\nolimits_{I\times J}(u)\, u
$$
$\rhd$
On utilise maintenant $1 \in \calD_r(u)$. Il y a une famille $(\lambda_{I,J})_{\#I=r,\#J=r}$
de scalaires telle que
$$
1 = \sum_{I,J}\, \lambda_{I,J}\, \det\nolimits_{I\times J}(u)
$$
On pose alors $v = \sum_{I,J} \lambda_{I,J} \, v_{I,J}$
pour obtenir $u \circ v \circ u = v$. 
\end{proof}

\subsubsection*{En présence d'une factorisation $\BW^r(u) = \Theta\, \nu$}

Nous schématisons encore une fois ce type de factorisation:
$$
\xymatrix @M= 0.4pc{ 
\BW^{r}(E) \ar[rd]_-{\nu} \ar[rr]^-{\BW^{r}(u)} & & 
\BW^{r}(F) \\
& \bfA \ar[ru]_-{\times \Theta} & 
}
$$
En désignant par $\Theta_I$ la composante de $\Theta$ sur $f_I$, nous utiliserons pour $x_1, \ldots, x_r \in E$:
$$
[u(x_1), \ldots, u(x_r)]\strut_I = \Theta_I\,\nu(x_1 \wedge \ldots \wedge x_r)
$$

\subsubsection*{Explication de texte du contenu du lemme qui suit}

Le lemme suivant prend en donnée une factorisation $\BW^r(u) =
\Theta\,\nu$ où $\Gr(\Theta) \geqslant 1$.  En conséquence,
$\calD_r(u) = \rmc(\nu)\rmc(\Theta)$ et, d'après la proposition~\ref
{StableRankFromFactorization}, $\calD_{r+1}(u) = 0$.
Son but est de fournir une \og sorte de présentation locale de $\Coker(u)$\fg{} au sens
suivant. \emph {Supposons $\nu(e_J)$ inversible} donc
$\DVect_r(u) = \bfA\Theta^\sharp$ et
$\calD_r(u) = \rmc(\Theta)$ est fidèle,
si bien que $u$ est de rang~$r$, donc $\Coker(u)$ est de rang $c := \dim F - r$.
L'apport du lemme (sous la clause $\nu(e_J)$ inversible) réside dans la somme directe:
$$
E = \Ker(u) \oplus \bigoplus_{j\in J} \bfA e_j
$$
La restriction $u'$ de $u$ à la somme directe de droite est donc injective de même
image que $u$:
$$
u' : \bigoplus_{j\in J} \bfA e_j \simeq  \bfA^r  \overset{\simeq}{\longrightarrow}
\Im u \subset F \simeq \bfA^{r+c}
$$
Ainsi $M := \Coker(u) = \Coker(u')$ est présenté par une \emph{injection} $\bfA^r \hookrightarrow \bfA^{r+c}$
et le sous-module de Fitting $\FittVect_c(M) \subset \BW^c(M)^\star$ est libre de rang 1.

\bigskip

On a vu en \ref{FactorisationEtPgcdFort} qu'un module $M$ de MacRae de
rang $c$ pouvait être présenté par $u : E \to F$ telle qu'en notant $r
= \dim F -c$, on ait une factorisation $\BW^r(u) = \Theta\nu$ avec
$\Gr(\nu) \ge 2$ et $\Gr(\Theta) \ge 1$.  En appliquant ce qui précède
au module $M := \Coker(u)$, on obtient localement (au sens $\Gr \ge
2$) une présentation de~$M$ par une \emph{injection} $\bfA^r
\hookrightarrow \bfA^{r+c}$ comme ci-dessus.

Ceci constitue une généralisation du cas $c=0$ énoncé
en~\ref{LocalSquarePresentation} (présentation locale carrée d'un
module de MacRae de rang 0) et du cas $c=1$
en~\ref{1ElementaryPresentation} (présentation locale $1$-élémentaire
d'un module de MacRae de rang $1$).

\bigskip

Tout le sel de l'affaire réside dans le quasi-projecteur $\pi_J$, objet du lemme
suivant. Essayons d'expliquer d'où il sort en prenant $r=2$. Comme on veut
$\Im\pi_J \subset \Ker u$, il est impératif de \emph {produire}
des vecteurs pertinents appartenant à $\Ker u$. Cette production est basée
sur le principe suivant (merci Cramer)
$$
\forall\ x_1,x_2,x \in E : \qquad\qquad
\nu(x_1\wedge x_2)x - \big(\nu(x\wedge x_2)x_1 + \nu(x_1\wedge x)x_2 \big)  \in \Ker u
\leqno (\star)
$$
Pour le voir, on utilise la proposition~\ref {StableRankFromFactorization} qui donne $\BW^{r+1}(u) = 0$,
ici $\BW^3(u) = 0$, ce qui permet d'appliquer le corollaire~\ref{ExtPowerNullity}, pour n'importe quelle partie $I$
de cardinal $r=2$
$$
[u(x_1),u(x_2)]\strut_I u(x) = [u(x),u(x_2)]\strut_I u(x_1) + [u(x_1),u(x)]\strut_I u(x_2) 
$$
Soit, en utilisant $\BW^2(u) = \Theta\,\nu$:
$$
\Theta_I\,\nu(x_1\wedge x_2) u(x) = \Theta_I\,\nu(x\wedge x_2) u(x_1)  +  \Theta_I\,\nu(x_1\wedge x) u(x_2)
$$
Cette égalité étant vérifiée pour toute composante $\Theta_I$, on peut, puisque $\Gr(\Theta) \ge 1$, simplifier
par $\Theta_I$:
$$
\nu(x_1\wedge x_2) u(x) = \nu(x\wedge x_2) u(x_1)  +  \nu(x_1\wedge x) u(x_2)
$$
Ce qui est exactement $(\star)$.

\begin{lem}[Le quasi-projecteur $\pi_J$ sous-jacent à une factorisation de $\BW^r(u)$]\label{ProjecteurLocal}
\leavevmode

Soit $u : E \to F$ et une factorisation $\BW^r(u) = \Theta\,\nu$ où
$\Gr(\Theta) \geqslant 1$. A une partie $J$ de cardinal~$r$ de
l'ensemble des indices de la base de $E$
$$
J = \{j_1 < \cdots < j_r\}
$$
on associe l'endomorphisme $\pi_J$ de $E$ défini par:
$$
\pi_J(x) \ = \ \nu(e_J)x \, -\,  
\sum_{\ell=1}^r \, 
\nu\big(e_{j_1} \wedge \cdots \wedge e_{j_{\ell-1}} \wedge x \wedge e_{j_{\ell+1}} \cdots \wedge e_{j_r} \big)
\, e_{j_\ell}
$$
On a alors les propriétés suivantes.

\begin {enumerate} [\rm i)]
\item  
\og Encadrement\fg{} de $\Ker u$:
$$
\Im \pi_J 
\ \subset \ 
\Ker u
\ \subset \ 
\Ker\big(\nu(e_J)\Id_E - \pi_J\big)
$$

\item
\og Encadrement\fg{} de $\displaystyle \bigoplus_{j \in J} \bfA e_j$:
$$
\Im\big(\nu(e_J)\Id_E - \pi_J\big)
\ \subset \ 
\bigoplus_{j \in J} \bfA e_j
\ \subset \ 
\Ker \pi_J
$$
\item
Si $\nu(e_J) = 1$, alors $\pi_J$ est un projecteur de $E$, d'image $\Ker(u)$,
de noyau $\displaystyle \bigoplus_{j \in J} \bfA e_j$.

\item
Si $\nu(e_J)$ est seulement inversible, alors
$$  
E = \Ker(u) \oplus \bigoplus_{j\in J} \bfA e_j
$$
\end {enumerate}  
\end{lem}

\begin{proof} \leavevmode

i)
$\rhd$
Inclusion $\Im \pi_J \subset \Ker u$.
Montrons pour $x\in E$ que $u\big(\pi_J(x)\big) = 0$.
D'après la proposition~\ref {StableRankFromFactorization}, on a $\BW^{r+1}(u) = 0$,
ce qui permet d'appliquer le corollaire~\ref{ExtPowerNullity}:
$$
\det\nolimits_{I \times J}(u) \  u(x) 
\ =\ 
\sum_{\ell=1}^r \, 
\big[u(e_{j_1}), \dots, u(e_{j_{\ell-1}}), u(x), u(e_{j_{\ell+1}}), \dots, u(e_{j_r}) \big]\strut_I \ 
u(e_{j_\ell})
$$
En utilisant $\BW^r(u) = \Theta \, \nu$, on obtient 
$$
\Theta_I \, \nu(e_J) \, u(x) 
\ =\ 
\Theta_I\, 
\sum_{\ell=1}^r \, 
\nu\big(e_{j_1} \wedge \cdots \wedge e_{j_{\ell-1}} \wedge x \wedge e_{j_{\ell+1}} \cdots \wedge e_{j_r} \big) \, 
u(e_{j_\ell})
$$
Ceci a lieu pour toute partie $I$ de cardinal $r$ et 
$\Gr(\Theta) \geqslant 1$, d'où 
$$
\nu(e_J) u(x)\ = \ \sum_{\ell=1}^r \, 
\nu\big(e_{j_1} \wedge \cdots \wedge e_{j_{\ell-1}} \wedge x \wedge e_{j_{\ell+1}} \cdots \wedge e_{j_r} \big)
\, u(e_{j_\ell})
$$
ce qui est une traduction de l'égalité $u\big(\pi_J(x)\big) = 0$ à montrer.

\smallskip
\noindent
$\rhd$
Inclusion $\Ker u \subset \Ker\big(\nu(e_J)\Id_E - \pi_J\big)$.  Soit
$x \in \Ker u$. Alors $\BW^{r}(u)(x \wedge z) = 0$ pour tout $z
\in \BW^{r-1}(E)$.  Comme $\Gr(\Theta) \geqslant 1$, on en
déduit $\nu(x \wedge z) = 0$.  Ainsi, dans la somme figurant dans
l'expression de~$\pi_J(x)$, chaque terme est nul. D'où $\pi_J(x) =
\nu(e_J) x$.

\medskip
ii)
L'inclusion $\Im\big(\nu(e_J)\Id_E - \pi_J\big) \subset \displaystyle\bigoplus_{j\in J} \bfA e_j$
résulte directement de la définition de $\pi_J$ puisque
$\pi_J(x) - \nu(e_J)x \in \displaystyle\bigoplus_{j\in J} \bfA e_j$.

\smallskip
\noindent
Pour montrer l'inclusion $\displaystyle \bigoplus_{j\in J} \bfA e_j \subset \Ker \pi_J$
i.e. $\pi_J(e_j) = 0$ pour $j \in J$, on réécrit la définition:
$$
\pi_J(e_j) \ = \ \nu(e_J)e_j \, -\,  
\sum_{k \in J} \, 
\pm \nu\big(e_j \wedge e_{J \setminus k}\big)\, e_k
$$
Dans la somme, chaque terme est nul, sauf celui pour $k = j$ égal à
$\nu(e_J) e_j$. En définitive $\pi_J(e_j) = 0$.

\medskip
iii)
Résulte des deux \og encadrements\fg.

\medskip
iv)
Remplacer $\pi_J$ par $\pi_J/\nu(e_J)$.
\end{proof}

\medskip

Lorsqu'on \og globalise\fg{} l'hypothèse ``$\nu(e_J)$ inversible'' en ``$\nu$ surjective'',
un certain nombre de résultats du lemme ne sont plus valides. Demeure cependant
le résultat suivant utilisé dans le th.~\ref{RacineIdealFactorisation}

\begin{prop} \label{FactorisationKerFacteurDirect}
Soit $u : E \to F$ et une factorisation $\BW^r(u) = \Theta\, \nu$ vérifiant
$\Gr(\Theta) \geqslant 1$ et $\nu$ surjective.
Alors $\Ker u$ est facteur direct dans $E$.
\end{prop}

\begin{proof}\leavevmode 
  
Comme $\nu$ est surjective, la famille $\big(\nu(e_J)\big)_{\#J=r}$
est comaximale.  Ecrivons ${\sum_{\#J = r} \alpha_J \, \nu(e_J) = 1}$
avec des $\alpha\strut_J \in \bfA$ et posons $\pi = \sum \alpha_J \,
\pi_J$ où $\pi_J$ est l'endomorphisme de $E$ défini dans le lemme
précédent~\ref{ProjecteurLocal} .

Vérifions la double inclusion 
$$
\Im \pi 
\ \subset \ 
\Ker u
\ \subset \ 
\Ker(\Id_E - \pi)
$$

\noindent
$\rhd$ D'après le lemme, on a $\Im \pi_J \subset \Ker u$, donc $\Im \pi \subset \Ker u$ par définition de $\pi$.

\noindent 
$\rhd$
Pour la deuxième inclusion, soit $x \in \Ker u$.
D'après le lemme, on a $\Ker u \subset \Ker\big(\nu(e_J)\Id_E - \pi_J\big)$, 
donc $\nu(e_J) x = \pi_J(x)$. En multipliant par $\alpha_J$ et en sommant, on obtient 
$$
\overbrace{\big(\sum_J \alpha_J \nu(e_J) \big)}^{1} \, x 
\ = \ 
\overbrace{\sum_J \alpha_J \pi_J(x)}^{\pi(x)}
\qquad\quad \text{\idest} \qquad\quad x = \pi(x)
$$
De la double inclusion, on tire $\Im \pi \subset \Ker
(\Id_E- \pi)$, prouvant que $\pi$ est un projecteur.  La double
inclusion se réduit alors à $\Im \pi = \Ker u = \Ker (\Id_E- \pi)$. En
conséquence, $\Ker u$ est facteur direct dans $E$.
\end{proof}

\medskip

Le th.~\ref{RacineIdealFactorisation} a également sollicité un autre résultat, conséquence
de la proposition suivante:

\begin{prop}\label{FactorisationImFacteurDirect}

Soit $u : E \to F$ et une factorisation $\BW^r(u) = \Theta\,\nu$ telle que
$\Gr(\Theta) \geqslant 1$ et $\Gr(\nu) \geqslant 1$. 
Alors les conditions suivantes sont équivalentes :

\begin{enumerate} [\rm i)]
\item
$\Im u$ est facteur direct dans $F$.

\item
$\nu$ est surjective et $\Theta$ est unimodulaire.

\item
$1 \in \calD_r(u)$. 
\end{enumerate}
Dans ce cas, $\Theta$ est une base de $\Im\BW^r(u)$.
\end{prop}

\begin{proof}\leavevmode

Comme $\BW^r(u) = \Theta\,\nu$, on a $\calD_r(u) = \rmc(\Theta)\rmc(\nu)$ donc
$1 \in \calD_r(u) \iff 1 \in \rmc(\Theta) \text{ et } 1 \in \rmc(\nu)$, d'où
l'équivalence $\rm iii)  \Longleftrightarrow ii)$.

\medskip

Montrons $\rm iii) \Rightarrow i)$. D'après la proposition~\ref{StableRankFromFactorization}, on a
$\calD_{r+1}(u) = 0$. Ainsi $u$ est de rang fort égal à $r$ et le fait que $\Im u$ soit
facteur direct dans $F$ découle de \ref{RangFort}.

\medskip

Terminons par $\rm i) \Rightarrow ii)$.
Montrons tout d'abord que $\Im\BW^{r}(u)$ est facteur direct dans $\BW^r(F)$. 
Pour le voir, on utilise la proposition~\ref{ImFacteurDirect} en considérant
un pseudo-inverse~$v$ pour~$u$ (\idest{} une application telle que $u\circ v \circ u = u$), 
ce qui fournit $\BW^{r}(v)$ comme pseudo-inverse de $\BW^{r}(u)$.
On termine en appliquant le lemme ci-dessous à
$$
M = \BW^r(E), \qquad N = \BW^r(F), \qquad f = \BW^r(u), \qquad
y = \Theta, \qquad \mu = \nu
$$

\begin{quote}
\it
Soit $f : M \rightarrow N$ une application linéaire entre deux modules
quelconques avec une factorisation $f = y \mu$ où $y \in N$, $\mu \in
M^\star$ et $\Ann(y) = 0$, $\Ann(\mu) = 0$.  Si $\Im f$ est facteur
direct dans~$N$, alors $\mu$ est surjective et $y$ est une base de
$\Im f$.  En particulier, si les modules $M$ et~$N$ sont libres, alors
$\mu$ et $y$ sont unimodulaires.
\end{quote}

Ecrivons $N = \Im f \oplus S$. 
Il existe $x_0 \in M$ et $s \in S$ tels que $y = f(x_0) + s = \mu(x_0)y + s$.
Nous allons montrer $s = 0$ donc $y=f(x_0)$ et $\mu(x_0)=1$.
L'expression $s = (1-\mu(x_0))\,y$ montre que $s \in \bfA y$
donc, pour tout $x \in M$, $\mu(x)s \in \bfA \mu(x)y = \bfA f(x)$.
Il s'ensuit que $\mu(x)s \in S \cap \Im f = \{0\}$
et donc $s=0$ puisque $\Ann(\mu) = 0$.

\smallskip
Une fois obtenu $s=0$, on en
déduit $y = f(x_0)$ et puisque $\Ann(y) = 0$, que $1 = \mu(x_0)$.
Bilan: $\mu$ est surjective, $\Im f = \bfA y$ et $y$ est une base de
$\Im f$.  Ainsi $N = \bfA y \oplus S$; la projection de $N$ sur $\bfA
y \simeq \bfA$ fournit une forme linéaire $\phi : N \to \bfA$ telle
que $\phi(y) = 1$, ce qui fait que $y$ est unimodulaire lorsque $N$
est libre.
\end{proof}


\section{Le théorème de proportionnalité et la structure multiplicative}
\label {sectionAnnexeStructureMultiplicative}

\subsection {Le théorème de proportionnalité}

Deux cadres (mathématiquement équivalents) sont possibles: celui de $\bfA^n$ muni de sa
base canonique ou celui un peu plus abstrait d'un module libre orienté
$(L, \bfe)$ de rang~$n$, qui permet la distinction entre~$L$ et son
dual~$L^\star$.  On a privilégié le produit intérieur droit et
l'isomorphisme Hodge droit (confer la section~\ref{ExteriorAlgebraDuality}),
mais le choix du côté droit n'a aucune vertu particulière par rapport à la gauche.

\medskip
Reprenons quelques éléments de la section~\ref{ExteriorAlgebraDuality}.
Dans le premier cadre, $\BW^\sbullet(\bfA^n)$ est muni de sa base
canonique $(e_I)$ ce qui définit un \og produit scalaire\fg{}
noté $\scp{\ }{\ }$ pour lequel cette base est orthonormale.
On désigne par $\Hd : \BW^\sbullet(\bfA^n) \to \BW^{n-\sbullet}(\bfA^n)$
l'isomorphisme de Hodge droit;
pour des dimensions $p,q$ complémentaires à $n$, $p+q = n$, il est
caractérisé par:
$$
\scp{\Hd(\bfx)}{\bfx'} = [\bfx\wedge\bfx'] \qquad
\bfx \in \BW^q(\bfA^n),\qquad \bfx' \in \BW^p(\bfA^n)
$$
où l'on rappelle que $[\bfx\wedge\bfx']$ désigne la coordonnée de $\bfx\wedge\bfx'$
dans la base $e_1 \wedge\cdots\wedge e_n$ de $\BW^n(\bfA^n) \simeq\bfA$.
Cet isomorphisme de Hodge vérifie $\Hd(e_I) = \varepsilon(I,\overline I)\, e_{\overline I}$.

\medskip
Dans le cadre d'un module libre orienté $(L,\bfe)$ de rang $n$, la forme
bilinéaire $\scp{\ }{\ } : \bfA^n \times \bfA^n \to \bfA$ est remplacée par
la dualité intrinsèque $L \times L^\star \to \bfA$ définie par $\scp{x}{\alpha} = \alpha(x)$.
Dualité qui s'étend à $\BW^d(L) \times \BW^d(L^\star)$ par:
$$
(x_1 \wedge \cdots \wedge x_d,\ \alpha_1 \wedge \cdots \wedge \alpha_d) \mapsto
\det(\alpha_i(x_j))
$$
L'orientation intervient seulement pour définir l'isomorphisme de Hodge droit:
$$
\Hd = \Hd^{\,\llcorner}_{\bfe^*} : \BW^{\sbullet}(L) \to \BW^{n-\sbullet}(L^\star)
\overset{\rm can.}{\simeq} \Big(\BW^{n-\sbullet}L\Big)^\star
\qquad
\text {qui réalise }
\bfx \mapsto \bfe^\star \intd \bfx = [\bfx\wedge\sbullet]_{\bfe}
$$
Pour alléger dans la suite, on utilisera $\bfx^\sharp$ au lieu de $\Hd(\bfx)$.

\begin{theo}[de proportionnalité] \label{ThProportionnalite}
On garde le contexte ci-dessus, en particulier $p+q = n$.

\smallskip
\noindent
(i) Enoncé \og détypé \fg 

Soient deux familles $u_1, \ldots, u_p$ et $v_1, \ldots, v_q$ de
vecteurs de~$\bfA^n$ qui sont orthogonales i.e. $\scp{u_i}{v_j} = 0$.
On pose $\bfu = u_1 \wedge \cdots \wedge u_p$ et $\bfv =
v_1 \wedge \cdots \wedge v_q$. Alors les deux vecteurs
$\bfu^\sharp$ et $\bfv$ de $\BW^q(\bfA^n)$ sont proportionnels.

\smallskip
\noindent
(ii) Enoncé (équivalent) plus abstrait.

Soit d'une part $p$ vecteurs $u_1, \ldots, u_p \in L$ et d'autre part
$q$ formes linéaires $\alpha_1, \ldots, \alpha_q : L \to \bfA$, le tout
vérifiant $\alpha_j(u_i) = 0$. On note $\bfu = u_1 \wedge \cdots \wedge u_p$.
Alors dans $\BW^q(L^\star)$, les deux $q$-vecteurs $\bfu^\sharp$ et
$\BW_{j=1}^q \alpha_j$ sont proportionnels.
\end{theo}

\index{théorème!de proportionnalité}

La preuve figure après les deux résultats auxiliaires suivants:
une formule de Plücker-Sylvester et son corollaire.
Dans le théorème ci-dessous, on énonce le résultat dans $\bfA^n$.
Pour $n$ vecteurs $z_1, \cdots, z_n$ de $\bfA^n$, la notation
$[z_1, \cdots, z_n]$ désigne le déterminant de $z_1, \cdots, z_n$ dans
la base canonique de $\bfA^n$.  On pourrait l'énoncer pour un module
libre orienté $(L, \bfe)$ de rang $n$, $[z_1, \cdots, z_n]$ désignant
cette fois $[z_1 \wedge \cdots \wedge z_n]_\bfe$.

\begin {theo}[Plücker-Sylvester] \label{PluckerSylvester} \leavevmode

Pour deux suites $\uz = (z_1, \cdots, z_n)$ et $\ux = (x_1, \cdots, x_q)$
de vecteurs de $\bfA^n$, on a l'égalité:
$$
[z_1, \dots, z_n]\,x_1\wedge\dots\wedge x_q =
\sum_{\#K=q}\, [\uz\xleftarrow{K}\ux]\,\BW_{k\in K}z_k
$$
où la somme porte sur les parties $K = \{k_1 < \ldots < k_q\} \subset \{1..n\}$ de cardinal $q$
et $\uz\xleftarrow{K}\ux$ désigne la suite de $n$ vecteurs de $\bfA^n$ obtenue
à partir de $(z_1, \ldots, z_n)$ en y remplaçant $z_{k_1}$ par $x_1$,
$z_{k_2}$ par $x_2$ etc. Par exemple, pour $n=3$, $q=2$, l'égalité
ci-dessus est la suivante:
$$
[z_1, z_2, z_3]\ x_1 \wedge x_2 =
[x_1, x_2, z_3]\ z_1 \wedge z_2 + [x_1, z_2, x_3]\ z_1 \wedge z_3 +
[z_1, x_1, x_2]\ z_2 \wedge z_3 
$$
\end {theo}

Remarque: pour $q=1$, le membre gauche est
$[z_1, z_2, \dots, z_n]\,x$ et la formule ci-dessus s'\'ecrit
$$
[z_1, z_2, \dots, z_n]\,x =
[x, z_2,\dots, z_n]\,z_1 \ +\ [z_1, x,\dots, z_n]\,z_2 
\ +\ \cdots \ +\ [z_1, z_2,\dots, x]\,z_n
$$
Ce n'est rien d'autre qu'une formule habituelle de Cramer.

\index{théorème!de Plücker-Sylvester}
\begin {proof} \leavevmode

Commençons par remarquer que l'égalité à démontrer est une identité algébrique en les
coordonnées des vecteurs qui y interviennent et que chaque membre est linéaire en chacun
des $z_i$ et $x_j$.

\medskip
\noindent
\emph{Cas particulier}. Supposons d'abord que $(z_1, \cdots, z_n)$ est une base de $\bfA^n$. Par linéarité
sur les $x_j$, il suffit de traiter le cas où chaque $x_j$ est pris parmi les
vecteurs de cette base:
$$
x_j = z_{k_j}  \qquad  1 \le j \le q
$$
Supposons deux des $k_j$ égaux. Alors les deux membres sont nuls: en effet, pour le
membre gauche c'est immédiat et pour le membre droit, chaque $[\uz\xleftarrow{K}\ux]$
est nul.

\smallskip
\noindent
Supposons les $k_j$ deux à deux distincts et posons $K' = \{k_1 < \cdots < k_q\}$.
Alors les termes de la somme intervenant dans le second membre sont 
nuls à l'exception de celui correspondant à $K = K'$, qui vaut
$$
[\uz\xleftarrow{K'} \ux]\, z_{k_1} \wedge \cdots \wedge z_{k_q} =
[\uz\xleftarrow{K'} \ux]\, x_1 \wedge\cdots\wedge x_q
$$
Mais comme on a l'égalité de suites
$(\uz\xleftarrow{K'} \ux) = (z_1, \cdots, z_n)$ quel que soit l'ordre des
$k_j$, le $q$-vecteur ci-dessus n'est autre que le premier membre
de l'égalité du théorème.

\smallskip
\noindent
\emph{Cas général}.
Puisqu'il s'agit d'une identité algébrique, on peut supposer que les
coordonnées $z_{i,i'}$ des~$z_i$ et $x_{j,j'}$ des $x_j$ sont des
indéterminées sur $\bbZ$.  On prend pour $\bfA$ le corps de fractions
rationnelles $\bbQ(z_{i,i'}, x_{j,j'})$, auquel cas $(z_1, \cdots,
z_n)$ est une base de $\bfA^n$ puisque $\det(z_1, \cdots, z_n) \ne 0$,
ce qui nous ramène au premier~cas.
\end {proof}

\begin {coro}
Soient $p,q$ avec $p+q = n$ et $u_1, \cdots, u_p$ des vecteurs de $\bfA^n$. On pose
$\bfu = u_1\wedge \cdots\wedge u_p$.
Alors pour deux $q$-vecteurs $\bfx, \bfy \in \BW^q(\bfA^n)$, le $q$-vecteur 
$[\bfu\wedge\bfx]\,\bfy - [\bfu\wedge\bfy]\,\bfx$ est combinaison linéaire de
$q$-vecteurs de la forme $w_1\wedge\cdots\wedge w_q$ avec $w_1 \in \{u_1, \cdots, u_p\}$:
$$
[\bfu\wedge\bfx]\,\bfy - [\bfu\wedge\bfy]\,\bfx \quad\in\quad
\sum_{j=1}^p u_j\,\wedge\,\BW^{q-1}(\bfA^n)
$$
\end {coro}

\begin {proof}
Puisque $[\bfu\wedge\bfx]\,\bfy - [\bfu\wedge\bfy]\,\bfx$ est bilinéaire en $(\bfx,\bfy)$,
on peut supposer $\bfx = x_1\wedge\cdots\wedge x_q$ et $\bfy = y_1\wedge\cdots\wedge y_q$
avec les $x_i,y_j \in \bfA^n$. Posons
$$
\bfw = [\bfu\wedge \bfx]\,\bfy, \qquad \bfw' = [\bfu\wedge \bfy]\,\bfx
$$
D'après la formule de Sylvester-Plücker (\ref{PluckerSylvester}) appliquée à
$$
(z_1, \cdots, z_p,\ z_{p+1}, \cdots, z_{n}) = (u_1, \cdots, u_p,\ x_1, \cdots, x_q)
$$
le $q$-vecteur
$\bfw = [\bfu\wedge\bfx]\,\bfy$ est une somme de termes de la forme
$a\,t_1 \wedge \cdots \wedge t_q$ avec $a \in \bfA$, les $t_k$ étant pris
dans $\{u_1, \cdots, u_p\} \cup \{x_1, \cdots, x_q\}$.
Dans chacun des termes $a\,t_1 \wedge \cdots \wedge t_q$, il y a un
$t_k$ égal à l'un des~$u_j$, à l'exception du terme correspondant dans
la formule (\ref{PluckerSylvester}) à la partie $K = \{p+1, \cdots, p+q\}$
qui sélectionne $(x_1, \cdots, x_q)$.
Ce terme exceptionnel est $[\bfu \wedge\bfy]\,\bfx = \bfw'$.  Ainsi la différence $\bfw
- \bfw'$ est une somme de $q$-vecteurs de la forme
$a\,t_1 \wedge \cdots \wedge t_q$, et dans chacun d'entre eux, l'un des
$t_k$ est un $u_j$ variable dans $\{ u_1, \cdots, u_p\}$.  Dans
$t_1 \wedge \cdots \wedge t_q$, on fait alors passer ce $t_k=u_j$ en tête
pour obtenir $\pm w_1 \wedge \cdots \wedge w_q$
avec $w_1 = u_j \in \{u_1, \cdots, u_p\}$.
\end {proof}



\begin {proof} [Preuve du théorème de proportionnalité \ref{ThProportionnalite}] \leavevmode

On doit démontrer que les deux $q$-vecteurs $\bfu^\sharp$ et $\bfv$ sont proportionnels
i.e. que pour tous $I,J \subset \{1..n\}$ vérifiant $\#I=\#J=q$, on a 
$\scp{\bfu^\sharp}{e_I}\scp{\bfv}{e_J} = \scp{\bfu^\sharp}{e_J}\scp{\bfv}{e_I}$.
Ou encore par définition de l'isomorphisme de Hodge droit
$$
[\bfu \wedge e_I] \scp{\bfv}{e_J} = [\bfu \wedge e_J] \scp{\bfv}{e_I}
$$
En posant
$$
\bfw = [\bfu \wedge e_I] e_J, \qquad \bfw' = [\bfu \wedge e_J] e_I
$$
nous devons démontrer que $\scp{\bfw-\bfw'}{\bfv} = 0$ sous le couvert
de l'hypothèse que chaque vecteur de $\{u_1, \cdots, u_p\}$ est orthogonal à
chaque vecteur de $\{v_1, \cdots, v_q\}$.

\smallskip
Nous allons faire mieux en prouvant que 
$\scp{\bfw-\bfw'}{\bfv}$ est combinaison linéaire des produits scalaires
$\scp{u_k}{v_j}$ pour $1\le k\le p$, $1 \le j \le q$. La technique
est donc analogue à celle de la remarque~\ref{rmqProportionnalite}
dans le cas particulier où $q=1$.

D'après le corollaire précédent, $\bfw - \bfw'$ est une combinaison
linéaire de $q$-vecteurs $w_1\wedge\cdots\wedge w_q$ avec
$w_1 \in \{u_1, \cdots, u_p\}$.
En conséquence, $\scp{\bfw-\bfw'}{\bfv}$ est une combinaison linéaire
de déterminants $\det\big(\scp{w_i}{v_j}_{1\le i,j\le q}\big)$.
Mais comme $w_1 \in \{u_1, \cdots, u_p\}$, le développement de ce déterminant
par rapport à la première ligne fournit le fait qu'il est combinaison
linéaire des $\scp{u_k}{v_j}$ pour $1\le k\le p$, $1 \le j \le q$.
En définitive, on obtient que  $\scp{\bfw-\bfw'}{\bfv}$ est une combinaison linéaire
des produits scalaires $\scp{u_k}{v_j}$.
\end {proof}

\subsection {L'existence de la structure multiplicative : le théorème de propagation}

\index{structure multiplicative d'un complexe de Cayley}%
\index{factorisation $\text{colonne}\times\text{ligne}$ (puissance extérieure d'une application linéaire)}%

Il s'agit ici de donner une preuve du théorème~\ref{Factorisation}. Nous devons
donc montrer, dans le contexte de ce théorème, l'existence de vecteurs 
$\Theta_k \in \bigwedge^{r_k}(F_{k-1})$ factorisant $\bigwedge^{r_k}(u_k)$,
le vecteur $\Theta_{n+1} = 1 \in \bigwedge^0(F_n)$ étant imposé. 
La preuve a lieu par récurrence sur $k$ en descendant jusqu'au degré homologique~1.
Vu la définition des rangs attendus d'un complexe, la preuve se ramène au
théorème de propagation ci-dessous dont nous précisons le contexte.

\medskip
\noindent
{\bf Contexte}:
On considère un petit complexe de modules libres dont on a indiqué les
dimensions en dessous:
$$
\begin {array} {ccccc}
E    &\xrightarrow{\quad u\quad} & F   & \xrightarrow{\quad v\quad} & G
\\
p+q  &                      & q+r &                      & r+s
\\
\end {array}
$$
On suppose seulement $v \circ u = 0$ sans demander l'exactitude (au milieu). Et on suppose disposer
de $\Theta_F \in \BW^q(F)$ et d'une forme linéaire $\nu_E$ sur $\BW^q(E)$
factorisant $\BW^q(u)$ i.e. $\BW^q(u) = \Theta_F\ \nu_E$:
$$
\xymatrix @R=0.5cm{
\BW^q(E) \ar[dr]_{\nu_E}\ar[rr]^{\BW^q(u)} && \BW^q(F) \\
             & \bfA\ar[ur]_{\times \Theta_F} \\
}
$$
En prenant l'idéal contenu dans l'égalité $\BW^q(u) = \Theta_F\ \nu_E$, on obtient
$\calD_q(u)  = \rmc(\Theta_F)\,\rmc(\nu_E)$ d'où
$$
\Gr(\Theta_F) \ge \Gr \calD_q(u), \qquad
\Gr(\nu_E) \ge \Gr \calD_q(u)
$$
Afin d'assurer le théorème de propagation qui suit, on impose $\Gr(\nu_E) \ge 1$ et
$\Gr(\Theta_F) \ge 2$.  Ceci est en accord avec le contexte du théorème~\ref{Factorisation}.

\noindent
Enfin, on fixe une orientation $\bff$ sur $F$ via une base de $\BW^{q+r}(F)$
et on note $\sharp = \sharp_\bff$ l'isomorphisme (de Hodge droit) associé:
$$
\sharp : \BW^{\sbullet}(F) \to \Big(\BW^{\dim F-\sbullet} F\Big)^\star
\qquad \hbox {défini par} \qquad
\Theta \mapsto \Theta^\sharp = [\Theta \wedge \sbullet]_\bff
$$


\begin {theo} [Propagation de la factorisation]\leavevmode
Dans le contexte précédent :
\begin {enumerate}[\rm i)]
\item
Il y a un unique $\Theta_G \in \BW^r(G)$ factorisant $\BW^r(v)$:
$$
\xymatrix @R=0.5cm{
\BW^r(F) \ar[dr]_{\Theta_F^\sharp}\ar[rr]^{\BW^r(v)} && \BW^r(G) \\
             & \bfA\ar[ur]_{\times \Theta_G} \\
}
$$
\item
On dispose de l'inégalité $\Gr(\Theta_G) \ge \Gr\calD_r(v)$.
\end{enumerate}
\end {theo}

\index{théorème!de propagation de la factorisation}%

\begin {proof}\leavevmode

i) Existence. \og Détypons les modules\fg{} en considérant des bases de $E, F, G$,
ce qui nous ramène à
$$
E = \bfA^{p+q}, \qquad F = \bfA^{q+r}, \qquad G = \bfA^{r+s}, \qquad
\hbox {
dont on note $(e_j)$, $(f_i)$ et $(g_k)$ les bases respectives}
$$
On choisit pour $\bff$ l'orientation associée à la base de $F$ si bien
que l'application $\sharp$ est
$$
\BW^q F\to \BW^r F, \qquad
\bfy \mapsto \bfy^\sharp = \sum_{\#I = r} [\bfy \wedge f_I] f_I 
$$
On note $U,V$ les matrices de $u,v$; on écrit l'hypothèse en disant
que $\BW^q U$ est le produit d'un vecteur-colonne et d'un vecteur-ligne:
$$
\BW^q U = \bfy \transpose{\bfx}, \qquad
\bfy \in \bfA^{\binom{q+r}{q}} \simeq \BW^q F, \quad
\bfx \in \bfA^{\binom{p+q}{q}} \simeq \BW^q E, \qquad
\Gr(\bfy) \ge 2, \quad \Gr(\bfx) \ge 1
$$
On doit montrer l'existence et l'unicité de $\bfz \in \bfA^{\binom{r+s}{r}} \simeq \BW^r G$
vérifiant $\BW^r V = \bfz \transpose{(\bfy^\sharp)}$.

Dans la suite, on réserve le symbole $J$ pour une partie
$J \subset \{1..p+q\}$ de cardinal $q$; et de même le symbole $I$
pour $I \subset \{1..r+s\}$ avec $\#I = r$.  On note $(U_j)$ les
\emph{colonnes} de $U$ et $(V_i)$ les \emph{lignes} de $V$.

\noindent
L'hypothèse $V \circ U = 0$ se traduit en disant, que dans $F
= \bfA^{q+r}$, chaque vecteur $U_j$ est \og orthogonal\fg{} à chaque
vecteur $V_i$.  D'après le théorème de
proportionnalité (\ref{ThProportionnalite}), 
dans $\BW^r F$, les $r$-vecteurs $\big(\BW_{j \in J} U_j\big)^\sharp$
et $\BW_{i \in I} V_i$ sont proportionnels.

\noindent
Quant à l'hypothèse $\BW^q U = \bfy \transpose{\bfx}$, elle se traduit
par des égalités $\BW_{j \in J} U_j = x_{\!_J}\,\bfy$. 

\noindent
Bilan: dans $\BW^rF$, les $r$-vecteurs $x_{\!_J}\,\bfy^\sharp$ et
$\BW_{i \in I} V_i$ sont proportionnels.  Comme $\Gr(\bfx) \ge 1$, on
en déduit que les $r$-vecteurs $\bfy^\sharp$ et $\BW_{i \in I} V_i$
sont proportionnels.  Puisque $\Gr(\bfy^\sharp) = \Gr(\bfy) \ge 2$, il
y a un unique scalaire $z\strut_I \in \bfA$ vérifiant
$$
\BW_{i \in I} V_i = z\strut_I\, \bfy^\sharp
$$
Ceci fournit un $\bfz = \sum_I z\strut_I g\strut_I \in \BW^r G$ tel que $\BW^r V
= \bfz \transpose{(\bfy^\sharp)}$.

\medskip
L'unicité est laissée au lecteur.

\medskip
\noindent
ii) On a $\BW^r v = \Theta_G \Theta_F^\sharp$. En prenant l'idéal contenu,  on
obtient $\calD_r(v) = \rmc(\Theta_G) \rmc(\Theta_F)$, a fortiori l'inclusion
$\calD_r(v) \subset \rmc(\Theta_G)$; l'inégalité sur les profondeurs s'ensuit.
\end {proof}

\newpage



\newpage

\newbox\toto
\newbox\tata
\newlength\largeurtoto
\newlength\largeurtata

\newcommand \ttt[3]{\setbox\tata=\hbox{#1}%
\ifdim\wd\tata<.15\textwidth\relax%
\setlength{\largeurtoto}{.78\textwidth}%
\setlength{\largeurtata}{.15\textwidth}%
\else%
\setlength{\largeurtoto}{.93\textwidth}%
\addtolength{\largeurtoto}{-\wd\tata}%
\setlength{\largeurtata}{\wd\tata}%
\fi%
\setbox\toto=\hbox{\parbox[b]{\largeurtoto}{\leftskip10pt\parindent-\leftskip\strut#2\dotfill\par}}%
\smallskip\noindent\mbox{%
\parbox[b][\ht\toto][t]{\largeurtata}{#1}%
\parbox[b]{\largeurtoto}{\leftskip10pt\parindent-\leftskip\strut#2\dotfill\par}%
\hspace{.02\textwidth}%
\parbox[b]{.05\textwidth}{\hfill#3}%
}%
\par}

\newcommand\NOTA[3]{\ttt{#2$\phantom{_{S}^{1}}$}{#3}{\pageref{#1}}}

\def\mou{plus 0.6pt minus 0.3pt}
\newcommand\CHAP[1]{\goodbreak\vskip4mm\mou\noindent{\bf #1}\par\nobreak\vskip1mm \mou\nobreak}


\section*{Index des notations de NotreResultant (\today)}
\addcontentsline{toc}{section}{\bf Index des notations}

Certaines notations identiques apparaissent dans des chapitres différents. Ceci est dû
au contexte particulier dans lequel elles sont définies et permet 
de faciliter une lecture indépendante.

\CHAP{\nameref{ChapAC1}}

\NOTA {NOTA01-muProduitInterieur} {$\sbullet\intd\mu$}
      {anti-dérivation à gauche $\BW^\sbullet(L)\to\BW^{\sbullet-1}(L)$ définie par $\mu\in L^\star$}

\NOTA {NOTA01-DualiteCanonique} {$\scp{\quad}{\quad}$}
      {dualité canonique en algèbre extérieure : $\scp{(x_1\wedge\cdots\wedge x_d}{\mu_1\wedge\cdots\wedge\mu_d}= \det\mu_i(x_j)$}
\NOTA {NOTA01-eI} {$e_I$} {$\bigwedge_{i\in I}e_i$}
\NOTA {NOTA01-OrientationDuale} {$\bfe^\star$} {orientation sur $L^\star$ déduite d'une orientation $\bfe$ sur $L$ libre de rang fini}

\NOTA {NOTA01-sharp} {$\sharp_\bfe$}
      {isomorphisme déterminantal pour $(L,\bfe)$, $\BW^d(L)\to\big(\BW^{n-d}L\big)^\star, \bfx\mapsto\oriented{\bfx\wedge\sbullet}_\bfe$}
\NOTA {NOTA01-epsilon} {$\varepsilon(i,j),\varepsilon(I,J)$} {$\varepsilon(i,j)$: signe étendu à valeurs dans $\{0,\pm1\}$,
                        $\varepsilon(I,J) = \prod_{i\in I, j \in J} \varepsilon(i,j)$}
      
\NOTA {NOTA01-ProduitInterieurDroit} {$\intd$}
      {produit intérieur droit $\BW^{q-p}(L) \xleftarrow {\intd} \BW^q(L)\times\BW^p(L^*)$}
\NOTA {NOTA01-contenu} {$\rmc(\cdots)$} {idéal contenu d'un polynôme, d'un vecteur}
\NOTA {NOTA01-Ann} {$\Ann(\cdots)$} {idéal annulateur d'un module, d'un élément d'un module, d'une suite}
      
\NOTA {NOTA01-KoszulDown} {$\rmK_\sbullet(\ua),\rmK_\sbullet(\ua\,;M)$} {complexe de Koszul descendant de la suite $\ua$}

\NOTA {NOTA01-varepsiloniI} {$\varepsilon_i(I)$} {nombre d'éléments de $I$ strictement plus petits que $i$}
\NOTA {NOTA01-HKoszulDown} {$\rmH_\sbullet(\ua),\rmH_\sbullet(\ua\,;M)$} {homologie Koszul de la suite $\ua$}
\NOTA {NOTA01-KoszulUp} {$\rmK^\sbullet(\ua),\rmK^\sbullet(\ua\,;M)$} {complexe de Koszul montant de la suite $\ua$}
\NOTA {NOTA01-HKoszulUp} {$\rmH^\sbullet(\ua),\rmH^\sbullet(\ua\,;M)$} {cohomologie Koszul de la suite $\ua$}

\NOTA {NOTA01-Sqrtua} {$\sqrt{\mathstrut{\ua}}$} {idéal racine de l'idéal $\langle\ua\rangle$ engendré par la suite $\ua$}
\NOTA {NOTA01-GrSuite} {$\Gr(\ua\,;M)\ge k$} {nullité des $k$ premiers groupes de cohomologie: $\rmH^j(\ua\,;M)=0$, $0\le j< k$}
\NOTA {NOTA01-GrIdeal} {$\Gr(\fa\,;M)\ge k$} {pour $\ua$ système de générateurs de l'idéal $\fa$: $\Gr(\ua\,;M)\ge k$}
\NOTA {NOTA01-Gr2} {$\Gr(\ua)\ge 2$} {profondeur $\ge 2$ via le critère $a_ib_j = a_jb_i$}
\NOTA {NOTA01-ExpSuite} {$\ub^e$} {$(b_1^e,\cdots,b_n^e)$ lorsque $\ub= (b_1,\cdots,b_n)$}

\NOTA {NOTA01-calD} {$\calD_r(u)$} {idéal déterminantiel d'ordre $r$ de l'application linéaire $u$}
\NOTA {NOTA01-VeroneseIdeal} {$\scrV_r$} {idéal de Veronese engendré par les monômes de degré $r$ sans facteur carré}

\CHAP{\nameref{ObjetsSuiteP}}

\NOTA {NOTA02-D} {$D=(d_1,\cdots,d_n)$} {format de degrés} 
\NOTA {NOTA02-uP} {$\uP = (P_1,\cdots,P_n)$} {système polynomial homogène de format $D$} 
\NOTA {NOTA02-bfAX} {$\bfA[\uX]$} {anneau de polynômes $\bfA[X_1,\cdots,X_n]$} 
\NOTA {NOTA02-uXD} {$\uX^D$} {jeu étalon $(X_1^{d_1}, \cdots, X_n^{d_n})$}
\NOTA {NOTA02-pXD} {$\pXD$} {jeu étalon généralisé $(p_1 X_1^{d_1},\dots, p_n X_n^{d_n})$}

\NOTA {NOTA02-bfkindetsPi} {$\bfk[\indetsPi]$} {anneau des coefficients du système générique au dessus de $\bfk$} 
\NOTA {NOTA02-delta} {$\delta$} {degré critique $\sum_i (d_i-1)$ attaché à $D$} 

\NOTA {NOTA02-Xemouton} {$X^{\emouton}$} {mouton-noir $X_1^{d_1-1}\cdots X_n^{d_n-1}$ symbolisé par \mouton} 
\NOTA {NOTA02-nabla} {$\nabla=\nabla_\uP$} {déterminant bezoutien $\det(\dsV)$ d'une matrice bezoutienne $\dsV$ de $\uP$} 
\NOTA {NOTA02-uPsat} {$\uPsat$} {idéal $\{F\in\bfA[\uX]\mid\exists e\, X_i^eF\in\langle\uP\rangle\ \forall i\}$,
                                 saturé de l'idéal $\langle\uP\rangle$}
\NOTA {NOTA02-hatdi} {$\widehat d_i$} {$\prod_{j \ne i} d_j$, attaché à $D = (d_1,\cdots,d_n)$} 

\NOTA {NOTA02-bfB} {$\bfB,\bfB_d$} {anneau-quotient $\bfA[\uX]/\langle\uP\rangle$ et sa composante homogène de degré $d$} 
\NOTA {NOTA02-Syl} {$\Syl=\Syl(\uP)$} {application de Sylvester de $\uP$ définie par $(U_1,\dots,U_n) \mapsto
                                       U_1P_1+\cdots+U_nP_n$}
\NOTA {NOTA02-Syld} {$\Syl_d=\Syl_d(\uP)$} {composante homogène de degré $d$ de $\Syl = \Syl(\uP)$}
\NOTA {NOTA02-bfBprime} {$\bfB',\bfB'_d$} {anneau-quotient $\bfA[\uX]/\langle\uP,\nabla\rangle$ et sa
                                         composante homogène de degré $d$}

\NOTA {NOTA02-ElimIdeal} {$\uPsat\cap\bfA$} {idéal d'élimination de $\uP$}
\NOTA {NOTA02-H0Cech} {$\vH^0_\uX(M)$} {module de cohomologie de \Cech{} du $\bfA[\uX]$-module $M$, en degré homologique $0$}
\NOTA {NOTA02-H0CechB} {$\vH^0_\uX(\bfB)$} {sous $\bfA[\uX]$-module gradué de $\bfB=\bfA[\uX]/\langle\uP\rangle$,
                        égal à $\uPsat/\langle\uP\rangle$}
\NOTA {NOTA02-H0CechB} {$\vH^0_\uX(\bfB)_d$} {sous $\bfA$-module de $\bfB_d=\bfA[\uX]_d/\langle\uP\rangle_d$,
                        égal à $\uPsat_d/\langle\uP\rangle_d$}

\NOTA {NOTA02-tau0} {$\tau_0$} {$\tau_0:(\bfB_\delta)^\star \to \ElimIdeal$, morphisme défini par $\mu \mapsto\mu(\nabla)$}
\NOTA {NOTA02-Div} {$\DivSeq(X^\alpha)$} {ensemble des indices de $D$-divisibilité du monôme $X^\alpha$}
\NOTA {NOTA02-Jex} {$\Jex_1,\Jex_2$} {idéaux monomiaux excédentaires (attachés au format $D$), $\Jex_1=\langle\uX^D\rangle$}
\NOTA {NOTA02-Jex12} {$\Jex_{1\setminus 2}$} {supplémentaire monomial de $\Jex_2$ dans $\Jex_1$ (attaché au format $D$)}
\NOTA {NOTA02-Jex12i} {$\Jex_{1\setminus2}^{(i)}$}
                    {sous-module monomial de $\bfA[\uX]$ de base les $X^\alpha$ divisibles uniquement par $X_i^{d_i}$}

\NOTA {NOTA02-rmKP} {$\rmK_\sbullet(\uP)$} {complexe de Koszul de $\uP=(P_1,\dots,P_n)$,
                       de terme $\rmK_k(\uP)=\BW^k(\bfA[\uX]^n)$}
\NOTA {NOTA02-partialkP} {$\partial_k(\uP)$} {différentielle $\rmK_k(\uP)\to\rmK_{k-1}(\uP)$ du complexe de Koszul de $\uP$}
\NOTA {NOTA02-eI} {$e_I$} {$\bigwedge_{i\in I}e_i$}
\NOTA {NOTA02-dI} {$d_I$} {$\sum_{i\in I} d_i$ (attaché au format $D$)}                       
\NOTA {NOTA02-rmKdP} {$\rmK_{\sbullet,d}(\uP)$} {composante homogène de degré $d$ du complexe de Koszul de $\uP$}

\NOTA {NOTA02-chid} {$\chi_d = \chi_d(D)$} {caractéristique d'Euler-Poincaré du complexe $\rmK_{\sbullet,d}(X^D)$,
                                            $\chi_d=\dim\bfA[\uX]_d/\langle\uX^D\rangle_d$}
\NOTA {NOTA02-Jex01} {$\Jex_{0\setminus 1}$} {supplémentaire monomial de $\Jex_1$ dans $\Jex_0=\bfA[\uX]$ (attaché au format $D$)}
\NOTA {NOTA02-sd} {$s_d=s_d(D)$} {$\dim\Jex_{1,d}$ (rang de l'application de Sylvester de $\uX^D$ en degré $d$)}
\NOTA {NOTA02-calD} {$\calD_r(u)$} {idéal déterminantiel d'ordre $r$ de l'application linéaire $u$}
\NOTA {NOTA02-rk} {$r_k = r_k(F_\sbullet)$}
                {rang attendu de la différentielle numéro $k$ du complexe $F_\sbullet$}
\NOTA {NOTA02-rkd} {$r_{k,d} = r_{k,d}(D)$}
                 {rang attendu de la différentielle numéro $k$ du complexe $\rmK_{\sbullet,d}(X^D)$}
\NOTA {NOTA02-sdelta} {$s_\delta$} {$\dim\Jex_{1,\delta}=\dim\bfA[\uX]_\delta-1$ (rang de $\Syl_\delta$)}

\CHAP{\nameref{DeuxCasEcole}}

\NOTA {NOTA03-omegares} {$\omegares=\omegaresP$} {forme linéaire fondamentale $\bfA[\uX]_\delta \to \bfA$
                                                  (aperçu de)}
\NOTA {NOTA03-mu} {$\mu$} {forme linéaire $\mu:\bfA[\uX]_\delta\to\bfA$ qualifiée de mirifique (en attendant $\omegares$)}

\NOTA {NOTA03-uxi} {$\uxi=(\xi_1,\cdots,\xi_n)$} {$\xi_i=\det(L_1,\cdots,L_{n-1},X_i)$ où chaque $L_j\in\bfA[\uX]$ est linéaire}
\NOTA {NOTA03-evalxi} {$\evalxi$} {forme linéaire $\bfA[\uX]_\delta \to \bfA$ d'évaluation en $\uxi$}

\NOTA {NOTA03-Omega} {$\Omega=\Omega_{P,Q}$} {constructeur $\Omega : \bfA[X,Y]_\delta\to\End(\bfA[X,Y]_\delta)$
                     (pour 2 polynômes en 2 variables)} 
\NOTA {NOTA03-omega} {$\omega=\omega_{P,Q}$} {forme linéaire $\omega : \bfA[X,Y]_\delta\to\bfA$
                                              (forme des mineurs signés de $\Syl_\delta$ car  $n=2$)}

\NOTA {NOTA03-W1delta+1} {$W_{1,\delta+1}$} {endomorphisme de $\bfA[X,Y]_{\delta+1}$
                         (pour 2 polynômes en 2 variables)}
\NOTA {NOTA03-epsilon} {$\varepsilon(i,j),\varepsilon(I,J)$} {$\varepsilon(i,j)$: signe étendu à valeurs dans $\{0,\pm1\}$,
                        $\varepsilon(I,J) = \prod_{i\in I, j \in J} \varepsilon(i,j)$}

\CHAP{\nameref{ChapJeuEtalonGeneralise}}

\NOTA {NOTA04-ResJeuEtalonGen} {$p_1^{\widehat d_1}\cdots p_n^{\widehat d_n}$}
      {résultant du jeu étalon généralisé}
      
\NOTA {NOTA04-omegaresJeuEtalonGen} {$p_1^{\widehat d_1 -1}\cdots p_n^{\widehat d_n -1} (X^\emouton)^\star$}
      {forme linéaire $\omegares$ du jeu étalon généralisé}

\NOTA {NOTA04-MacRaeBd} {$\MacRae(\bfB_d)$} {idéal de $\bfA$, invariant de MacRae
                                             de $\bfB_d=\bfA[\uX]_d/\langle\uP\rangle_d$ pour $d\ge\delta+1$}
\NOTA {NOTA04-MacRaeBdelta} {$\MacRaeVect(\bfB_\delta)$} 
                            {sous-module du $\bfA$-dual $\bfB_\delta^\star$, invariant de MacRae
                              de $\bfB_\delta=\bfA[\uX]_\delta/\langle\uP\rangle_\delta$}

\NOTA {NOTA04-Nnd} {$\bbN^n_d$} {$\{\alpha\in \bbN^n \mid |\alpha|=d\}$}
\NOTA {NOTA04-kappa} {$\kappa$} {mécanisme de sélection $\kappa:\bbN^n_d\rightarrowtail\{1..n\}$}  

\NOTA {NOTA04-ordreTordu}{$<_\sigma$} {ordre sur $\{1..n\}$ tordu par $\sigma\in\mathfrak S_n$:
                                      $\sigma(1)<_\sigma\cdots<_\sigma \sigma(n)$}

\NOTA {NOTA04-minDiv} {$\minDiv(X^\alpha)$}
                  {plus petit indice de $D$-divisibilité du monôme $X^\alpha$, cf aussi p.~\pageref{NOTA05-minDiv}}

\NOTA {NOTA04-sminDiv} {$\sminDiv(X^\alpha)$}
                        {plus petit indice au sens $<_\sigma$ de $D$-divisibilité du monôme $X^\alpha$}

\NOTA {NOTA04-detW1d} {$\det W_{1,d}(\uP)$} {mineur plein de $\Syl_d(\uP)$ pour $d\ge\delta+1$ (sélection $\minDiv$)}
\NOTA {NOTA04-detW1dsigma} {$\det W^\sigma_{1,d}(\uP)$}
                        {analogue à $\det W_{1,d}$ pour le mécanisme de sélection $\sminDiv$}

\NOTA {NOTA04-Rd} {$\calR_d=\calR_d(\uP)$} {générateur normalisé de $\MacRae(\bfB_d)$ pour $d\ge\delta+1$}
\NOTA {NOTA04-bd} {$b_d=b_d(\uP)$} {cofacteur de $\calR_d$ dans $\det W_{1,d}(\uP)$ pour $d\ge\delta+1$}
\NOTA {NOTA04-bdsigma} {$b_d^\sigma=b_d^\sigma(\uP)$} {cofacteur de $\calR_d$ dans $\det W^\sigma_{1,d}(\uP)$ pour $d\ge\delta+1$}

\NOTA {NOTA04-Dvect} {$\DVect_{\ell-1}(u)$}
              {sous-module de $F^\star$ engendré par les formes déterminantales de $u:E\to F\simeq \bfA^\ell$}

\NOTA {NOTA04-DvectSyldelta} {$\DVect_{s_\delta}(\Syl_\delta(\uP))$} 
               {sous-module de $\bfA[\uX]_\delta^\star$ engendré par les formes déterminantales de $\Syl_\delta(\uP)$}
               
\NOTA {NOTA04-omega} {$\omega=\omega_\uP$} {forme déterminantale de $\Syl_\delta(\uP)$ (mécanisme de sélection $\minDiv$)}

\NOTA {NOTA04-omegasigma} {$\omega^\sigma=\omega^\sigma_\uP$}
                          {forme déterminantale de $\Syl_\delta(\uP)$ (mécanisme de sélection $\sminDiv$)}

\NOTA {NOTA04-bdelta} {$b_\delta=b_\delta(\uP)$} {scalaire cofacteur dans $\omega = b_\delta\,\omegares$} 
\NOTA {NOTA04-bdeltasigma} {$b_\delta=b_\delta^\sigma(\uP)$} {scalaire cofacteur dans $\omega^\sigma = b_\delta^\sigma\,\omegares$} 

\CHAP{\nameref{ChapObjetsSylvester}}

\NOTA {NOTA05-minDiv} {$\minDiv(X^\alpha)$}
                  {plus petit indice de $D$-divisibilité du monôme $X^\alpha$, cf aussi p.~\pageref{NOTA04-minDiv}}

\NOTA {NOTA05-piJ1} {$\pi_{\Jex_1}$} {projection $\bfA[\uX]\twoheadrightarrow\Jex_1$
                                     parallèlement au supplémentaire monomial de $\Jex_1$}
\NOTA {NOTA05-W1d} {$W_{1,d}(\uP)$} {endomorphisme de $\Jex_{1,d}$ attaché à la sélection $\minDiv$}

\NOTA {NOTA05-maxDiv} {$\maxDiv(X^\alpha)$} {plus grand indice de $D$-divisibilité du monôme $X^\alpha$}

\NOTA {NOTA05-Jex1di} {$\Jex_{1,d}^{(i)}$} {sous-module de $\Jex_{1,d}$ de base les $X^\alpha$ 
                                           vérifiant $\minDiv(X^\alpha) = i$} 

\NOTA {NOTA05-piM} {$\pi_\calM$} {projection $\bfA[\uX]\twoheadrightarrow\calM$
                                 parallèlement au supplémentaire monomial de $\calM$}
\NOTA {NOTA05-WM} {$W_\calM(\uP)$} {endomorphisme induit-projeté sur $\calM\subset\Jex_1$ de $W_1(\uP)$ }
                                       
\NOTA {NOTA05-W2d} {$W_{2,d}(\uP)$} {endomorphisme de $\Jex_{2,d}$ attaché à la sélection $\minDiv$}

\NOTA {NOTA05-Jexh} {$\Jex_h$} {idéal monomial engendré par les $\prod_{i\in I} X_i^{d_i}$ pour $\#I=h$}
\NOTA {NOTA05-XDI} {$X^{D(I)}$} {$\prod_{i\in I} X_i^{d_i}$}

\NOTA {NOTA05-Wh} {$W_{h,d}(\uP)$} {endomorphisme induit-projeté sur $\Jex_{h,d}$ de $W_1(\uP)$}

\NOTA {NOTA05-Omega} {$\Omega=\Omega_\uP$}
                     {$\Omega:\bfA[\uX]_\delta\to\End(\bfA[\uX]_\delta)$ pilotée par $\minDiv$}
                     
\NOTA {NOTA05-omega} {$\omega=\omega_\uP$} {$\omega:\bfA[\uX]_\delta\to\bfA$, $F \mapsto \det(\Omega(F))$}

\NOTA {NOTA05-WMsigma} {$W_\calM^\sigma(\uP)$} {analogue à $W_M(\uP)$ pour la sélection $\sminDiv$}

\NOTA {NOTA05-sminDiv} {$\sminDiv(X^\alpha)$}
                        {plus petit indice au sens $<_\sigma$ de $D$-divisibilité du monôme $X^\alpha$}

\NOTA {NOTA05-Omegasigma} {$\Omega^\sigma=\Omega^\sigma_\uP$}
                          {$\Omega^\sigma:\bfA[\uX]_\delta\to\End(\bfA[\uX]_\delta)$ pilotée par $\sminDiv$}

\NOTA {NOTA05-omegasigma} {$\omega^\sigma=\omega^\sigma_\uP$}
                          {$\omega^\sigma:\bfA[\uX]_\delta\to\bfA$, $F \mapsto \det(\Omega^\sigma(F))$}

\NOTA {NOTA05-M^{(i)}} {$\calM^{(i)}$} {sous-module de $\calM$ de base les $X^\alpha$
                                       vérifiant $\minDiv(X^\alpha) = i$}

\NOTA {NOTA05-Jexhdi} {$\Jex_{h,d}^{(i)}$} {sous-module de $\Jex_{h,d}$ de base les $X^\alpha$ 
                                           vérifiant $\minDiv(X^\alpha) = i$} 

\NOTA {NOTA05-Sigma2} {$\Sigma_2$} {toute partie de $\fS_n$ telle que $\{\sigma(n)\mid\sigma\in\Sigma_2\}=\{1..n\}$}
\NOTA {NOTA05-Sigma3} {$\Sigma_3$} {toute partie de $\fS_n$ telle que $\{\sigma(\{n-1,n\})\mid\sigma\in\Sigma_3\}=\calP_2(\{1..n\})$}
\NOTA {NOTA05-Sigmah} {$\Sigma_h$} {définition analogue à celles de $\Sigma_2,\Sigma_3$}

\CHAP{\nameref{ChapJeuCirculaire}}

\NOTA {NOTA06-UnitPoint}{$\Un$} {$\Un = (1:1:\cdots:1)$, point projectif unité}

\NOTA {NOTA06-bsp-bsq}{$\bsp, \bsq$} {$\bsp=(p_1,\cdots,p_n)$, $\bsq=(q_1,\cdots,q_n)$ dans
                                      le contexte du jeu $\pXD + \bsq\,\uR$}

\NOTA {NOTA06-wMiterateur}{$w_\calM(\uR)$}{itérateur sur Monômes($\bfA[\uX]_d$), piloté par $\minDiv$,
                  défini par $\calM\subset\Jex_{1,d}$ et le jeu~$\uR$}

\NOTA {NOTA06-wMiterateur}{$w_{1,\delta}$} {$w_\calM(\uR)$ pour $\calM = \Jex_{1,\delta}$ et $\uR$ le jeu circulaire}

\NOTA {NOTA06-bbNhmg}{$\bbN^n_d,\ \bbZ^n_d$}{$\bbN^n_d = \{\alpha\in\bbN^n \mid|\alpha| = d\}$
                      avec $|\alpha| = \sum_i\alpha_i$. Idem pour $\bbZ^n_d$}
\NOTA {NOTA06-emouton}{$\emouton$}{$\emouton=(d_1-1,\cdots,d_n-1)$, exposant du mouton-noir}

\NOTA {NOTA06-strate}{$\calS_t$} {strate horizontale des monômes à hauteur $t$ du puits
                                  (graphe d'un itérateur)}

\NOTA {NOTA06-Oalpha}{$\calO_\alpha,\overline{\calO_\alpha}$}{\og orbite\fg{} sous $w_{1,\delta}$ de
                      $X^\alpha\in\bfA[\uX]_\delta$ et son complémentaire dans $\bfA[\uX]_\delta$}

\NOTA {NOTA06-w2diterateur} {$w_{2,d}$} {$w_\calM(\uR)$ pour $\calM = \Jex_{2,d}$ et $\uR$ le jeu circulaire}

\NOTA {NOTA06-MatJac1Q}{$\uR$}{$(R_1,\cdots,R_n)$ jeu monomial de format $D$}
\NOTA {NOTA06-MatJac1Q}{$\Jac_\Un(\uX^D-\uR)$}{matrice de $\bbM_n(\bbZ)$ dite attachée à $\uR$, souvent notée $A$}
\NOTA {NOTA06-MatJac1Q}{$\ell_1,\dots,\ell_n$}{les $n$ lignes $\ell_i=\exposant(X_i^{d_i})-\exposant(R_i)$ de la matrice
                                               $\Jac_\Un(\uX^D-\uR)$}
\NOTA {NOTA06-MatJac1Q}{$r_1,\dots,r_n$}{les $n$ mineurs principaux (d'ordre $n-1$) de la matrice attachée à $\uR$}

\NOTA {NOTA06-jeuR1} {$\uR^{(1)}$} {jeu $(X_1^{d_1-1}X_2, X_1^{d_2-1}X_3, \dots, X_1^{d_n-1}X_1)$. 
                               Egalement $\uR^{(2)},\dots,\uR^{(n)}$ sur ce modèle}

\NOTA {NOTA06-ovq} {$\overline {q}_i$} {$\overline{q}_i=-q_i$ dans le contexte d'un jeu $\pXD+\bsq\,\uR$}

\CHAP{\nameref{ChapBgenerique}}

\NOTA {NOTA07-Bdelta}{$\bfB_\delta$}{$\bfA$-module, composante homogène de
          $\bfB=\bfA[\uX]/\langle\uP\rangle$ de degré (critique) $\delta$}

\NOTA {NOTA07-omegaBdelta}{$\overline\omega$}{$\overline\omega : \bfB_\delta\to\bfA$
            forme linéaire $\omega$ passée au quotient modulo $\langle\uP\rangle_\delta$}

\NOTA {NOTA07-Bxj}{$\bfB_{x_j}$}{localisé de $\bfB=\bfA[\uX]/\langle\uP\rangle$ en $x_j$,
                                 classe de $X_j$ modulo $\langle\uP\rangle$}

\NOTA {NOTA07-pij}{$p_{ij}$}{coefficient en $X_j^{d_i}$ du polynôme $P_i$}
\NOTA {NOTA07-pij}{$\bfA'_j$}{sous-$\bfk$-algèbre de $\bfA$ engendrée
                              par les coefficients autres que les $(p_{i,j})_{i=1,\cdots,n}$}

\CHAP{\nameref{ChapFittingVectoriel}}

\noindent
$E,F$ sont des modules libres, $u:E\to F$ une application linéaire,
et $M$ un module présenté par $u$
\smallskip

\NOTA {NOTA08-Dru} {$\calD_r(u)$} {idéal déterminantiel d'ordre $r$ de $u$ (matrice ou
                                  application linéaire)}
                                  
\NOTA {NOTA08-FkM} {$\calF_k(M)$} {idéal de Fitting d'ordre $k$ de $M$}

\NOTA {NOTA08-MacRaeScalaire} {$\MacRae(M)$} {invariant de MacRae de $M$, où $M$ est de rang $0$\par 
                               idéal de $\bfA$ engendré par un pgcd fort de $\calF_0(M)$}

\NOTA {NOTA08-DVect} {$\DVect_r(u)$}{sous-module de $\BW^c(F)^\star$ où $c=\dim F-r$\par 
   image de $\Im\BW^r(u)$ par un isomorphisme déterminantal
   $\BW^r(F)\overset{\simeq}{\longrightarrow}\BW^c(F)^\star$}

\NOTA {NOTA08-FittVect1} {$\FittVect_1(M)$}{sous-module de $M^\star$ où $M$ est de rang $1$\par 
            passage au quotient de $\DVect_{\dim F-1}(u) \subset F^\star$}

\NOTA {NOTA08-FittVectc} {$\FittVect_c(M)$}{sous-module de $\BW^c(M)^\star$ où $M$ est de rang $c$\par 
      passage au quotient de $\DVect_{\dim F-c}(u)\subset\BW^c(F)^\star$}

\NOTA {NOTA08-MacRaeVect}{$\MacRaeVect(M)$}{sous-module monogène de $\BW^c(M)^\star$ où $M$ est de rang $c$\par
                          module engendré par $\vartheta$ pgcd fort de $\FittVect_c(M)$ dans $\BW^c(M)^\star$}

\NOTA {NOTA08-Cramerm0M} {$\Cramer_{m_0}(M)$}
        {sous-module de $M^\star$\par
        constitué des $\alpha \in M^\star$ telles que $\alpha(m_0)\sbullet=\alpha(\sbullet)m_0$}

\CHAP{\nameref{ChapStructureMultiplicative}}

\NOTA {NOTA09-FreeComplex} {$(F_\sbullet,u_\sbullet)$} {complexe libre avec $u_i:F_i\to F_{i-1}$ pour $1\le i\le n$
                           et $F_{n+1}=0$ }
\NOTA {NOTA09-rk} {$(r_0,\ldots,r_{n+1})$} {rangs attendus du complexe $F_\sbullet$ avec $r_{n+1}=0$ et
                                           $r_k + r_{k+1}=0$}
\NOTA {NOTA09-r0} {$c=r_0$} {caractéristique d'Euler-Poincaré du complexe $F_\sbullet$}

\NOTA {NOTA09-bfe} {$\bfe_k$}{orientation sur le terme $F_k$}
\NOTA {NOTA09-sharp} {$\sharp_k$}{isomorphisme déterminantal $\sharp_k:\BW^{r_{k+1}}(F_k) \simeq\BW^{r_k}(F_k)^\star$}
\NOTA {NOTA09-Theta} {$(\Theta_k)_{1\le k\le n+1}$} {système de factorisation $\Theta_k\in\BW^{r_k}(F_{k-1})$
                        avec $\BW^{r_k}(u_k)=\Theta_k\Theta_{k+1}^\sharp$ où $\sharp=\sharp_k$}
\NOTA {NOTA09-mu} {$\mu$} {forme $c$-linéaire alternée $\mu\in\BW^c(F_0)^\star$
                            (générateur du déterminant de Cayley de~$F_\sbullet$)}

\NOTA {NOTA09-CayleyVect} {$\CayleyVect(F_\sbullet)$}{déterminant de Cayley du complexe $F_\sbullet$ (sous-module
                           libre de rang 1 de $\BW^c(F_0)^\star$)}

\NOTA {NOTA09-Det} {$\Det(F_\sbullet),\ \rho$} {déterminant de $F_\sbullet$ et l'isomorphisme $\rho:\Det(F_\sbullet)\simeq\CayleyVect(F_\sbullet)$}

\NOTA {NOTA09-fDk} {$\fD_k$} {idéal déterminantiel $\calD_{r_k}(u_k)$ dans un complexe libre $(F_\sbullet,u_\sbullet)$}
\NOTA {NOTA09-fBk} {$\fB_k$} {idéal $\rmc(\Theta_k)$, contenu du $r_k$-vecteur $\Theta_k$}

\CHAP{\nameref{ChapMacRaeForP}}

\NOTA {NOTA10-omegares-Rd} {$\omegares=\omegaresP$} {forme linéaire fondamentale normalisée $\bfA[\uX]_\delta\to\bfA$
                         (définition générale)}

\NOTA {NOTA10-omegares-Rd} {$\calR_d=\calR_d(\uP)$} {générateur normalisé de $\MacRae(\bfB_d)$ pour $d\ge\delta+1$
                                           (définition générale)}

\NOTA {NOTA10-fbd} {$\fb_d=\fb_d(\uP)$}{idéal cofacteur de $\calR_d(\uP)$ dans $\calF_0(\bfB_d)=\calD_{s_d}(\Syl_d)$
                                       pour $d\ge\delta+1$}
\NOTA {NOTA10-fbd} {$\fb_\delta=\fb_\delta(\uP)$}{idéal cofacteur de $\omegaRes{\uP}$ dans $\DVect_{s_\delta}(\Syl_\delta)$}

\NOTA {NOTA10-bdsigma} {$b_d=b_d(\uP)$} {cofacteur de $\calR_d$ dans $\det W_{1,d}(\uP)$ pour $d\ge\delta+1$
                                    et de $\omegares$ dans $\omega$ pour $d=\delta$}

\NOTA {NOTA10-bdsigma} {$b^\sigma_d=b^\sigma_d(\uP)$} {cofacteur tordu par $\sigma\in\fS_n$}

\NOTA {NOTA10-MonomialBezMat} {$\dsV_\uR$} {Matrice bezoutienne du jeu monomial $\uR$
                                            (un seul coefficient non nul par colonne)}

\NOTA {NOTA10-produitXJRJ} {$X_J,\ R_J$} {$X_J=\prod_{j\in J}X_j$, $R_J=\prod_{j\in J}R_j$ pour $J\subset\{1..n\}$}

\NOTA {NOTA10-jeuPred} {$\uPred$} {jeu creux (peu de coefficients) vérifiant $\Gr(\omegaRes{\uPred})\ge 2$}

\NOTA {NOTA10-jeuRj} {$\uR^{(j)}$} {jeu monomial $(X_j^{d_1-1}X_2,X_j^{d_2-1}X_3,\dots,X_j^{d_n-1}X_1)$}

\CHAP{\nameref{ChapMacRaeDefResultant}}

\NOTA {NOTA11-Psigma}{$\uP_\sigma$}{$\uP_\sigma=(P_{\sigma(1)}, \dots, P_{\sigma(n)})$ où $\sigma\in\fS_n$}
\NOTA {NOTA11-Fou}{$F\circ u$}{$F\circ u=F(Y_1,\ldots,Y_n)$ avec $Y_i = \sum_ju_{ij} X_j$ où
                   $F=F(\uX)\in\bfA[\uX]$ et $u\in\bbM_n(\bfA)$}
\NOTA {NOTA11-Pou}{$\uP^u$}{$\uP^u = (P_1\circ u, \dots, P_n\circ u)$ où $u\in\bbM_n(\bfA)$}
\NOTA {NOTA11-Discr}{$\Discr(F)$}{$\Res(\partial_1F, \dots, \partial_nF)$ où
                     $F$ est un polynôme homogène en $n$ variables}

\NOTA {NOTA11-Disc}{$\Disc(F)$}{$\Discr(F)/d^\sbullet$ avec $d^\sbullet$ puissance ad-hoc de $d=\deg(F)$}

\CHAP{\nameref{ChapElimination}}

\NOTA {NOTA12-FoQ}{$F\circ\uQ$}{$F\circ\uQ=F(Q_1,\ldots,Q_n)$ où $F, Q_i\in\bfA[\uX]$ et les $Q_i$ de même degré}
\NOTA {NOTA12-FXq}{$F(\uX^q)$}{$F(\uX^q) = F(X_1^q,\ldots, X_n^q)$}

\CHAP{\nameref{ChapWW}}

\NOTA {NOTA13-hderniers}{$\hderniers$}{suite $(n-h+2, \ldots, n)$ des $h-1$ derniers éléments de $\{1..n\}$}

\NOTA {NOTA13-maxDiv} {$\maxDiv(X^\alpha)$} {plus grand indice de $D$-divisibilité du monôme $X^\alpha$}
\NOTA {NOTA13-M[[j]]} {$\calM\double{j}$}
               {sous-module de $\calM$ de base les $X^\alpha \in \calM$ tels que $\maxDiv(X^\alpha) = j$}
\NOTA {NOTA13-Mssup} {$\calM^\ssup$}
               {sous-module de $\calM$ de base les $X^\alpha \in \calM$ tels que $\alpha_j>d_j$ où $j=\maxDiv(X^\alpha)$}
\NOTA {NOTA13-Meq} {$\calM^\eq$}
               {sous-module de $\calM$ de base les $X^\alpha \in \calM$ tels que $\alpha_j=d_j$ où $j=\maxDiv(X^\alpha)$}
\NOTA {NOTA13-Whdeq} {$W^\eq_{h,d}(\uP)$} {endomorphisme $W_\calM(\uP)$ où $\calM=\Jex_{h,d}^\eq$}
\NOTA {NOTA13-Whdssup} {$W^\ssup_{h,d}(\uP)$} {endomorphisme $W_\calM(\uP)$ où $\calM=\Jex_{h,d}^\ssup$}

\NOTA {NOTA13-Lasth}{$\Last_h(X^\alpha)$}{suite $(\alpha_{m+1},\dots,\alpha_n)$ des $h-1$ derniers
                           exposants de $X^\alpha$ où $m=n-(h-1)$}
\NOTA {NOTA13-LasthM}{$\Last_h(\calM)$}{$\{\Last_h(X^\alpha) \mid X^\alpha \in \calM\}$}

\NOTA {NOTA13-McalL}{$\calM^{(\calL)}$}{sous-module de $\calM$ de base les
                       $X^\alpha\in\calM \mid \Last_h(X^\alpha)=\calL$ ($\calL\in\bbN^\hderniers$)}

\NOTA {NOTA13-End}{$\End(X^\alpha)$}{$(\alpha_j, \dots, \alpha_n)$ où $j = \maxDiv(X^\alpha)$}
\NOTA {NOTA13-EndM}{$\End(\calM)$}{$\{ \End(X^\alpha) \mid X^\alpha \in \calM \}$ où $\calM \subset \Jex_2$}
\NOTA {NOTA13-McalE} {$\calME$}{sous-module de $\calM$ de base les
                       $X^\alpha\in\calM \mid \End(X^\alpha)=\calE$}

\NOTA {NOTA13-Jex1o} {$\Jex_1^\circ$}{sous-module de $\Jex_1$ de base les $X^\alpha$ avec $\alpha_n=0$}
\NOTA {NOTA13-Jex2=} {$\Jex_2^\eq$}{sous-module de $\Jex_2$ de base les $X^\alpha$ avec $\alpha_j=d_j$
                                 où $j=\maxDiv(X^\alpha)$}
\NOTA {NOTA13-W1o} {$W_{1,d}^\circ(\uP)$}{endomorphisme $W_\calM(\uP)$ où $\calM=\Jex^\circ_{1,d}$}
\NOTA {NOTA13-W2=} {$W_{2,d}^\eq(\uP)$}{endomorphisme $W_\calM(\uP)$ où $\calM=\Jex^\eq_{2,d}$}

\NOTA {NOTA13-Shift}{$\Shift(X^\alpha)$}{$X_1^{\alpha_1} \cdots X_{j-2}^{\alpha_{j-2}}
                     X_{j-1}^{\alpha_{j-1}+\alpha_j}X_j^{\alpha_{j+1}} \cdots X_{n-1}^{\alpha_n}$
                     où $j = \maxDiv(X^\alpha)$}

\NOTA {NOTA13-maxDivSh}{$\maxDivSh(X^\beta)$}
                       {$\max \big(k \in \{2..n\} \mid \beta_{k-1} \geqslant d_k \big)$} 

\NOTA {NOTA13-Jex1sh} {$\Jex_1^\sh$}{un certain sous-module monomial de $\Jex_1^\circ$, cf. la définition}
\NOTA {NOTA13-W1sh} {$W_1^\sh(\uP)$}{induit-projeté de $W_1(\uP)$ sur $\Jex_1^\sh$}

\CHAP{\nameref{ComplexeDecompose}}


$\varphi$ isomorphisme monomial normalisateur, guide sur lequel on s'aligne
$\xymatrix @R = 0.2pt{
\Mmac_{1+\bullet} &\Mmac_{\bullet}\ar@{-->}[ldd]|{\simeq}_{\varphi} \\
   \oplus       &\oplus \\
\Smac_{1+\bullet} &\Smac_{\bullet} \\
}$

\smallskip

\NOTA {NOTA14-partialkP}{$\partial_k = \partial_k(\uP)$}
      {différentielle $\rmK_k\to\rmK_{k-1}$ de degré homologique $k$ du complexe de Koszul de $\uP$}

\NOTA {NOTA14-partialkP}{$\partial_{k,d}=\partial_{k,d}(\uP)$}
      {composante homogène de degré $d$ de $\partial_k(\uP)$}

\NOTA {NOTA14-MmacSmac}{$\Mmac_\sbullet,\Smac_\sbullet$}
      {décomposition de Macaulay $\rmK_k = \Mmac_k\oplus\Smac_k$ du complexe de Koszul $\rmK_\sbullet(\uX^D)$}

\NOTA {NOTA14-betak}{$\beta_k=\beta_k(\uP)$}
      {induit-projeté de $\partial_k(\uP)$, $\beta_k:\Smac_k \to \Mmac_{k-1}$, cf. la fenêtre d'observation
       page \pageref{FenetreObservation}}
\NOTA {NOTA14-betak}{$\beta_{k,d} = \beta_{k,d}(\uP)$}
      {composante homogène de degré $d$ de $\beta_k$}

\NOTA {NOTA14-varphi}{$\varphi$} {$\varphi=\beta_\sbullet(\uX^D)^{-1} : \Mmac_\sbullet\to\Smac_{\sbullet+1}$,
                                  $X^\beta e_J \mapsto (X^\beta/X_i^{d_i}) e_i\wedge e_J$ où $i=\minDiv(X^\beta)$}

\NOTA {NOTA14-Bk}{$B_k = B_k(\uP)$}
      {endomorphisme de $\Mmac_{k-1}$ défini par $\beta_k(\uP)\circ\beta_k(\uX^D)^{-1}$}
      
\NOTA {NOTA14-Bk}{$B_{k,d}=B_{k,d}(\uP)$} {composante homogène de degré $d$ de $B_k(\uP)$}
      
\NOTA {NOTA14-DeltakdP}{$\Delta_{k,d}=\Delta_{k,d}(\uP)$}
      {$\det B_{k,d} = \Delta_{k,d} \Delta_{k+1,d}$}

\NOTA {NOTA14-MmacSmacsigma}{$\Mmac^\sigma_\sbullet,\Smac^\sigma_\sbullet$}
      {décomposition de Macaulay $\rmK_k = \Mmac^\sigma_k\oplus\Smac^\sigma_k$ de $\rmK_\sbullet(\uX^D)$
      tordue par $\sigma\in\fS_n$}

\NOTA {NOTA14-varphisigma}{$\varphi^\sigma$} {$\beta^\sigma_\sbullet(\uX^D)^{-1} : \Mmac^\sigma_\sbullet\to\Smac^\sigma_{\sbullet+1}$,
                                  $X^\beta e_J \mapsto (X^\beta/X_i^{d_i}) e_i\wedge e_J$ où $i=\sminDiv(X^\beta)$}

\NOTA {NOTA14-S0d}{$\Smac_{0,d}$}{supplémentaire monomial de $\Jex_{1,d}$, de base les $X^\alpha$
                                 avec $\alpha_i<d_i$ pour tout $i$}
\NOTA {NOTA14-S0d}{$\calS_{0,d}$}{base monomiale de $\Smac_{0,d}$, de cardinal $\chi_d$,
                                  caractéristique d'Euler-Poincaré de $\rmK_{\sbullet,d}$}

\NOTA {NOTA14-Detd}{$\Det_d=\Det_{d,\uP}$}
      {$\Det_d : \AXdSod \to \bfA$ ou $\bfB_d^{\calS_{0,d}} \to \bfA$, forme $\chi_d$-linéaire alternée} 

\NOTA {NOTA14-Omegad} {$\bsOmega_d = \bsOmega_{d,\uP}$}
                {$\bsOmega_d : \AXdSod \to \End(\bfA[\uX\rbrack_d)$}

\NOTA {NOTA14-bsomega}{$\bsomega_d=\bsomega_{d,\uP}$}{$\bsomega_d=\det \bsOmega_d$}

\NOTA {NOTA14-Mk-1i}{$\Mmac^{(i)}_{k-1}\subset \Mmac_{k-1}$}{base: $\{X^\beta e_J\mid\minDiv(X^\beta)=i\}$.
                     Usage: $\poids_{P_i}(\det B_{k,d}) = \dim\Mmac^{(i)}_{k-1,d}$}

\CHAP{\nameref{ChapBW}}

\NOTA {NOTA15-ekh}{$e_{k,h}$}{$e_{k,h}=\binom{h-2}{h-k}$, défini pour tous $h,k$,
       utilisé avec $\text{second indice}\ge\text{premier}\ge 1$}
\NOTA {NOTA15-Pontpsi}{Pont $\psi$}{$\bfA[\uX]$-morphisme $\psi:\rmK_\sbullet(\uX^D)\to \bfA[\uX]$
                       réalisant $e_J \mapsto X^{D(J)}$}
\NOTA {NOTA15-MJ}{$\calM[J]\subset\calM$}{sous-module monomial de base les $X^\beta e_J\in\calM$}
\NOTA {NOTA15-MmacJ}{$\Mmac[J]\subset\Mmac_{\#J}$}{sous-module monomial du module de Macaulay $\Mmac_{\#J}$,
                     de base les $X^\beta e_J\in\Mmac_{\#J}$}

\NOTA {NOTA15-BPcalM}{$B_\calM(\uP)$}{induit-projeté de $B(\uP)$ sur $\calM$
      (sous-module monomial d'un module de Macaulay)}

\NOTA {NOTA15-RevE}{$\RevE$}{suite obtenue en ordonnant l'ensemble $E$ de manière décroissante}
\NOTA {NOTA15-Ind}{$\Ind(J;E)$}{ensemble des indices de $J$ (contenu dans $E$) dans la suite $\RevE$}
\NOTA {NOTA15-Elt}{$\Elt(R;E)$}{partie de $E$ constituée des éléments d'indices $R$ dans $\RevE$
                               où $R\subset\{1...\#E\}$}

\NOTA {NOTA15-MmacR}{$\Mmac^R\subset \Mmac_{\#R}$}{sous-module de base les $X^\beta e_J$
                      tels que $\#J=\#R$ et $\Ind(J;\DivSeq(X^\beta X^{D(J)}))=R$}

\NOTA {NOTA15-JexR}{$\Jex^R$}{synonyme de $\Jex_{1+\max(R)}$}
\NOTA {NOTA15-psiR}{$\psi^R$}{$\psi^R:\Mmac^R \overset{\simeq}{\longrightarrow}\Jex^R\overset{\rm def.}{=}\Jex_{1+\max(R)}$,
                              isomorphisme restriction du pont $\psi$}         

\NOTA {NOTA15-JexRJ}{$\Jex^R[J]$}{sous-module de l'idéal $\Jex_{1+\max(R)}$, image de $\Mmac^R[J]$
                     par le pont $\psi$ (où $\#J=\#R$)}

\NOTA {NOTA15-BPR}{$B^R(\uP)$}{endomorphisme induit-projeté de $B(\uP)$ sur $\Mmac^R$}

\NOTA {NOTA15-fkh}{$f_{k,h}$}{$f_{k,h}=\binom{h-3}{h-k}$, défini pour tous $h,k$,
       utilisé avec $\text{second indice}\ge\text{premier}\ge1$}

\CHAP{\nameref{ChapSylvesterHybride}}

\smallskip
\noindent
Ne pas oublier que la plupart du temps le système $\uP$ est implicite.
\medskip

\NOTA {NOTA16-XbetaPlusUn} {$\uX^{\beta+\Un}$} {$(X_1^{\beta_1+1}, \cdots, X_n^{\beta_n+1})$}
\NOTA {NOTA16-Vbeta} {$\dsV_\beta$} {$\dsV_\beta\in\bbM_n(\bfA[\uX])$, homogène, vérifiant
                                    $[\uP]=[\uX^{\beta+\Un} ].\dsV_\beta$ pour $|\beta| < \min(D)$}
\NOTA {NOTA16-nablabeta} {$\nabla_\beta$} {$\nabla_\beta = \det(\dsV_\beta)$ pour $|\beta| < \min(D)$}

\NOTA {NOTA16-sw-fonction} {$f^{\rm sw}$} {$f^{\rm sw}\in\bfB^{\calS_{0,d}}_d$ est défini, pour $f\in\BdSodprime$,
       par $X^\alpha\mapsto f(X^{\emouton-\alpha})$ où $d'=\delta-d$}

\NOTA {NOTA16-swDetd} {$\swDet_{d,\uP}$} {$\BdSodprime\to \bfA$ défini par $f\mapsto\Det_d(f^{\rm sw})$ où $d'=\delta-d$}
\NOTA {NOTA16-swbsOmegad} {$\swbsOmega_{d,\uP}$} {$\bfA[\uX]_d^{\calS_{0,d'}}\to \End(\bfA[\uX]_d)$ défini par
                   $f\mapsto\bsOmega_d(f^{\rm sw})$ où $d'=\delta-d$}
\NOTA {NOTA16-swbsOmegasigmad} {$\bsOmega^{\sigma\,\rm sw}_{d,\uP}$} {$\bfA[\uX]_d^{\calS_{0,d'}}\to\End(\bfA[\uX]_d)$
       défini par $f\mapsto\bsOmega^\sigma_d(f^{\rm sw})$ où $d'=\delta-d$, $\sigma\in\fS_n$}
\NOTA {NOTA16-tau0} {$\tau_0$} {$\tau_0:(\bfB_\delta)^\star \to \ElimIdeal$ définie par $\mu \mapsto\mu(\nabla)$}
\NOTA {NOTA16-Evee} {$E^\vee$}{dual gradué du $\bfA[\uX]$-module gradué $E$}
\NOTA {NOTA16-HnCechuX} {$\vH^n_\uX(\bfA[\uX])$} {$n$-ème module de cohomologie de {\Cech} de $\uX$ sur $\bfA[\uX]$}

\CHAP{\nameref{SectionBezoutMorley}}

\smallskip
\noindent
Ci-dessous, le système $\uP$ est implicite et $d,d'$ sont liés par $d+d'=\delta$ (degré critique).
\medskip

\NOTA {NOTA17-AXY} {$\bfA[\uX,\uY]$}{$\bfA[X_1,\dots,X_n, Y_1,\dots, Y_n]$}
\NOTA {NOTA17-AXYpq} {$\bfA[\uX,\uY]_{p,q}$}{composante bihomogène de bidegré $(p,q)$ de $\bfA[\uX,\uY]$}
\NOTA {NOTA17-Bez} {$\Bez(\uX,\uY)$}{déterminant bezoutien du système (implicite) $\uP$}
\NOTA {NOTA17-mho}{$\mho_{d,\protect\uP}$}{$\mho_{d,\protect\uP} : \bfA[\uX,\uY]_{d,d'}
      \rightarrow\End\big(\bfA[\uX]_d\oplus\Jex_{1,d'}^\star\big)$, cf la définition \og géante\fg{} \ref{DefGeante}}
\NOTA {NOTA17-varpi}{$\varpi_{d,\uP}(F)$} {$\varpi_d(F) = \det\big(\mho_d(F)\big)$}
\NOTA {NOTA17-varpidres}{$\varpidRes{\uP}$}{$\bfA[\uX,\uY]_{d,d'} \to \bfA$,
                         $F\mapsto \varpi(F)/(\det W_{2,d}\det W_{2,d'})$}

\NOTA {NOTA17-bezendo}{$\bezendo(p,q)$}{$\bezendo(p,q)\in\End(\bfA[x]_{e-1})$ (``matrice'' de Bezout)
                       où $p,q\in\bfA[x]_{\le e}$}

\CHAP{\nameref{ChapSeries}}

\NOTA {NOTA18-XbarDI} {$\overline X^{D(I)}$} {$\prod_{i\in I}(1-X_i^{d_i})$}
\NOTA {NOTA18-XDI} {$X^{D(I)}$} {$\prod_{i\in I} X_i^{d_i}$}
\NOTA {NOTA18-fNum} {$\scrN_f$} {$\scrN_f(S) = (1-X_1)\cdots(1-X_n)\times S$
                     (\og numérateur fine-grading\fg{} de la série $S$)}
\NOTA {NOTA18-rNum} {$\scrN_r$} {$\scrN_f(S)= (1-t)^n\times S$
                    (\og numérateur raw-grading\fg{} de la série $S$)}
\NOTA {NOTA18-pi0} {$\pi_0(I)$}{$\pi_0(f,I)=\prod_{i\in I}f(i)$ où $f:F\to\bfR$, $I\subset F$
                   ($F$ ensemble fini, $\bfR$ anneau)}
\NOTA {NOTA18-pi1} {$\pi_1(I)$}{$\pi_1(f,I)=\prod_{i\in I}\big(1-f(i)\big)$}
\NOTA {NOTA18-JexI} {$\Jex_I$} {idéal monomial $\langle X_i^{d_i},\ i\in I\rangle$ de $\bfA[\uX]$}
\NOTA {NOTA18-brickI} {$\brick_I$}{$\bigoplus_{X^\beta \notin \Jex_I} \bfA X^\beta$ (supplémentaire
                       monomial de $\Jex_I$)}

\renewcommand\indexname{En guise d'index des termes, théorèmes, etc.
\protect\addcontentsline{toc}{section}{\bf Index terminologique (en cours).}}
\printindex
\end{document}